\newcount\ForArxivOrBLMS
\documentclass{amsart}
\usepackage{euscript}
\usepackage{amssymb}
\def\maxeuler{-1/24}%
\def\mineuler{-8}%
\def\withmax{$\Pi_{3,25}$, $\Pi_{3,28}$, $\Pi_{3,31}$ and $\Pi_{3,40}$}%
\def\withmin{$\Pi_{24,1}$}%
\def\numcpt{766}%
\def\numnoncpt{228}%
\def\numrightangled{101}%
\def\NumRegularWeyls{9}
\def\RegularWeyls{$\Pi_{3,1}$, $\Pi_{3,14}$, $\Pi_{3,15}$, $\Pi_{4,1}$, $\Pi_{4,2}$, $\Pi_{4,17}$, $\Pi_{6,1}$, $\Pi_{6,2}$ and $\Pi_{6,3}$}
\def\NumIdealWeyls{9}
\def\IdealWeyls{$\Pi_{3,1}$, $\Pi_{3,14}$, $\Pi_{3,15}$, $\Pi_{4,2}$, $\Pi_{4,6}$, $\Pi_{4,17}$, $\Pi_{6,3}$, $\Pi_{6,16}$ and $\Pi_{6,17}$}
\def\hardlatticerecog{\Pi_{4,86}}%
\def\veryhardlatticerecog{\Pi_{5,32}}%
\def\numgenusnotenough{14}
\def\numgenusalmostenough{12}
\def\GenusAlmostEnough{$\Pi_{4,16}$, $\Pi_{5,43}$, $\Pi_{6,9}$, $\Pi_{6,25}$, $\Pi_{6,42}$, $\Pi_{7,20}$, $\Pi_{8,35}$, $\Pi_{8,41}$, $\Pi_{8,52}$, $\Pi_{9,12}$, $\Pi_{10,5}$ and $\Pi_{12,13}$}
\def\sharedAA{$3:1\sim14\sim15$}%
\def\sharedAB{$3:3\sim6$}%
\def\sharedAC{$3:4\sim8$}%
\def\sharedAD{$3:5\sim7\sim9$}%
\def\sharedAE{$3:10\sim11\sim21$}%
\def\sharedAF{$3:12\sim13$}%
\def\sharedAG{$3:16\sim19\sim33$}%
\def\sharedAH{$3:17\sim18\sim20\sim36\sim38\sim43$}%
\def\sharedAI{$3:22\sim23\sim24$}%
\def\sharedAJ{$3:25\sim28\sim31\sim40$}%
\def\sharedAK{$3:26\sim27\sim29\sim35\sim39\sim41$}%
\def\sharedAL{$3:30\sim32\sim34\sim37\sim42\sim44$}%
\def\sharedAM{$4:2\sim17$}%
\def\sharedAN{$4:7\sim27$}%
\def\sharedAO{$4:8\sim37\sim39\sim41$}%
\def\sharedAP{$4:10\sim11$}%
\def\sharedAQ{$4:12\sim18$}%
\def\sharedAR{$4:13\sim20$}%
\def\sharedAS{$4:19\sim52$}%
\def\sharedAT{$4:21\sim24\sim57\sim60\sim61$}%
\def\sharedAU{$4:25\sim62$}%
\def\sharedAV{$4:26\sim29\sim31\sim34$}%
\def\sharedAW{$4:28\sim30\sim32\sim33\sim35\sim36$}%
\def\sharedAX{$4:38\sim40\sim42\sim43\sim81$}%
\def\sharedAY{$4:46\sim79$}%
\def\sharedAZ{$4:47\sim49$}%
\def\sharedBA{$4:48\sim50$}%
\def\sharedBB{$4:51\sim53\sim56$}%
\def\sharedBC{$4:55\sim58\sim59$}%
\def\sharedBD{$4:64\sim69$}%
\def\sharedBE{$4:65\sim67\sim68\sim70$}%
\def\sharedBF{$4:71\sim74$}%
\def\sharedBG{$4:72\sim75\sim77$}%
\def\sharedBH{$4:73\sim76$}%
\def\sharedBI{$4:78\sim80\sim82\sim84\sim85\sim87$}%
\def\sharedBJ{$4:83\sim86$}%
\def\sharedBK{$5:1\sim5$}%
\def\sharedBL{$5:4\sim7\sim12\sim16\sim42\sim48$}%
\def\sharedBM{$5:8\sim43$}%
\def\sharedBN{$5:13\sim17$}%
\def\sharedBO{$5:19\sim21$}%
\def\sharedBP{$5:20\sim22\sim23$}%
\def\sharedBQ{$5:31\sim34$}%
\def\sharedBR{$5:37\sim40$}%
\def\sharedBS{$5:44\sim50$}%
\def\sharedBT{$6:8\sim37$}%
\def\sharedBU{$6:11\sim38$}%
\def\sharedBV{$6:30\sim35$}%
\def\latticeHH{L_{16.7}}%
\def\latticeHHcount{12}%
\def\latticeHI{L_{16.9}}%
\def\latticeHIcount{12}%
\def\latticeHJ{L_{22.4}}%
\def\latticeHJcount{12}%
\def\latticeHK{L_{31.7}}%
\def\latticeHKcount{12}%
\def\latticeHL{L_{7.7}}%
\def\latticeHLcount{14}%
\def\latticeHM{L_{155.1}}%
\def\latticeHMcount{19}%
\def\latticeHN{L_{123.8}}%
\def\latticeHNcount{42}%
\def\latticeHO{L_{142.20}}%
\def\latticeHOcount{42}%
\def\latticeHP{L_{251.3}}%
\def\latticeHPcount{74}%
\def\latticeHQ{L_{16.13}}%
\def\latticeHQcount{390}%
\newdimen{\twowidth}
\newdimen{\halftwo}
\def\slashtwo{2\llap{\hbox to 0pt{\hss$|$\hss}\kern\halftwo}}
\newdimen{\threewidth}
\newdimen{\halfthree}
\def\slashthree{3\llap{\hbox to 0pt{\hss$|$\hss}\kern\halfthree}}
\newdimen{\inftywidth}
\newdimen{\halfinfty}
\def\slashinfty{\infty\llap{\hbox to 0pt{\hss$|$\hss}\kern\halfinfty}}
\def\onebar{\bar{1}}
\def\zerobar{\bar{0}}
\def\myAonezero{\Pi_{3,23}}
\def\myAoneI{\Pi_{3,2}}
\def\myAoneII{\Pi_{3,1}}
\def\myAoneIII{\Pi_{5,31}}
\def\myAtwozero{\Pi_{3,18}}
\def\myAtwoI{\Pi_{4,8}}
\def\myAtwoII{\Pi_{4,2}}
\def\myAtwoIII{\Pi_{8,4}}
\def\myAthreezero{\Pi_{3,7}}
\def\myAthreeI{\Pi_{4,3}}
\def\myAthreeII{\Pi_{6,3}}
\def\myAthreeIII{\Pi_{12,4}}
\def\myBone{\Pi_{4,7}}
\def\myBtwo{\Pi_{4,1}}
\def\myBthree{\Pi_{6,1}}
\def\myBfour{\Pi_{6,2}}
\def\myAoneIzerobar{\Pi_{3,24}}
\def\myAtwoIzerobar{\Pi_{3,17}}
\def\myAtwoIIonebar{\Pi_{4,38}}
\def\myAtwoIonebar{\Pi_{3,29}}
\def\myAtwozeroonebar{\Pi_{3,27}}
\def\myAthreeIzerobar{\Pi_{3,5}}
\def\myAthreeIIonebar{\Pi_{4,36}}
\def\myAthreeIonebar{\Pi_{3,34}}
\def\myAthreezeroonebar{\Pi_{3,32}}
\def\myAfourIzerobar{\Pi_{3,15}}
\def\myAfourIIonebar{\Pi_{4,39}}
\def\myAfourIonebar{\Pi_{3,38}}
\def\myAfourzeroonebar{\Pi_{3,36}}
\def\myAxi{\Pi_{5,34}}
\def\GNa{\Pi_{3,24}}
\def\GNb{\Pi_{3,22}}
\def\GNc{\Pi_{3,38}}
\def\GNd{\Pi_{3,23}}
\def\GNe{\Pi_{3,43}}
\def\GNf{\Pi_{3,17}}
\def\GNg{\Pi_{3,20}}
\def\GNh{\Pi_{4,27}}
\def\GNi{\Pi_{4,79}}
\def\GNj{\Pi_{3,36}}
\def\GNk{\Pi_{3,5}}
\def\GNl{\Pi_{3,2}}
\def\GNm{\Pi_{3,9}}
\def\GNn{\Pi_{3,15}}
\def\GNo{\Pi_{4,39}}
\def\GNp{\Pi_{4,59}}
\def\GNq{\Pi_{3,18}}
\def\GNr{\Pi_{3,14}}
\def\GNs{\Pi_{3,7}}
\def\GNt{\Pi_{4,37}}
\def\GNu{\Pi_{4,7}}
\def\GNv{\Pi_{4,46}}
\def\GNw{\Pi_{4,41}}
\def\GNx{\Pi_{4,58}}
\def\GNy{\Pi_{3,1}}
\def\GNz{\Pi_{5,34}}
\def\GNaa{\Pi_{4,8}}
\def\GNab{\Pi_{4,55}}
\def\GNac{\Pi_{4,1}}
\def\GNad{\Pi_{4,6}}
\def\GNae{\Pi_{4,17}}
\def\GNaf{\Pi_{8,11}}
\def\GNag{\Pi_{4,3}}
\def\GNah{\Pi_{5,38}}
\def\GNai{\Pi_{4,2}}
\def\GNaj{\Pi_{5,30}}
\def\GNak{\Pi_{6,1}}
\def\GNal{\Pi_{6,45}}
\def\GNam{\Pi_{6,22}}
\def\GNan{\Pi_{7,13}}
\def\GNao{\Pi_{6,17}}
\def\GNap{\Pi_{5,31}}
\def\GNaq{\Pi_{6,2}}
\def\GNar{\Pi_{6,16}}
\def\GNas{\Pi_{8,31}}
\def\GNat{\Pi_{6,3}}
\def\GNau{\Pi_{7,12}}
\def\GNav{\Pi_{8,45}}
\def\GNaw{\Pi_{8,46}}
\def\GNax{\Pi_{8,25}}
\def\GNay{\Pi_{9,7}}
\def\GNaz{\Pi_{9,16}}
\def\GNba{\Pi_{9,1}}
\def\GNbb{\Pi_{10,21}}
\def\GNbc{\Pi_{10,22}}
\def\GNbd{\Pi_{10,11}}
\def\GNbe{\Pi_{11,8}}
\def\GNbf{\Pi_{12,8}}
\def\GNbg{\Pi_{8,4}}
\def\GNbh{\Pi_{12,4}}
\long\def\ArxivOrBLMS#1#2{\ifodd\ForArxivOrBLMS {#2}\else{#1}\fi}

\def\numsystems{994}
\def\nummatrices{317906}
\def\numposdefinite{5}
\def\numranktwo{9}
\def\numDwithnoWeylvector{310179}
\def\numDwithWeylvector{7713}
\def\numenlargements{61811}
\def\numclosedstepone{722}
\def\numextensionsstepone{10178}
\def\numclosedsteptwo{2446}
\def\numextensionssteptwo{5354}
\def\numclosedlaststep{280}
\def\numclosedchains{21831}

\def\numwiththreesimpleroots{44}
\def\numwithfoursimpleroots{87}
\def\numwithfivesimpleroots{53}
\def\numwithsixsimpleroots{58}
\def\numwithsevensimpleroots{21}
\def\numwitheightsimpleroots{55}
\def\numwithninesimpleroots{18}
\def\numwithtensimpleroots{33}
\def\numwithelevensimpleroots{18}
\def\numwithtwelvesimpleroots{33}
\def\numwiththirteedsimpleroots{15}
\def\numwithfourteensimpleroots{29}
\def\numwithfifteensimpleroots{39}
\def\numwithsixteensimpleroots{71}
\def\numwithseventeensimpleroots{85}
\def\numwitheightteensimpleroots{98}
\def\numwithnineteensimpleroots{78}
\def\numwithtwentysimpleroots{59}
\def\numwithtwentyonesimpleroots{38}
\def\numwithtwentytwosimpleroots{34}
\def\numwithtwentythreesimpleroots{18}
\def\numwithtwentyfoursimpleroots{10}

\def\hypcell#1{\bigl(\begin{smallmatrix}0&#1\\#1&0\end{smallmatrix}\bigr)}

\def\isomorphism{\cong}
\def\iso{\isomorphism}
\def\g{\mathfrak{g}}

\def\Aut{\mathop{\rm Aut}\nolimits}

\def\orthogonalgroup{{\rm O}}
\def\PGL{{\rm PGL}}
\def\semidirect{\rtimes}
\def\Atilde{\tilde{A}}
\def\a{\alpha}

\def\defn#1{{\it#1}}
\def\spanof#1{\langle#1\rangle}
\def\sset{\subseteq}
\def\Z{\mathbb{Z}}
\def\Q{\mathbb{Q}}
\def\R{\mathbb{R}}

\def\perp{\bot}

\def\set#1#2{\{#1{{}\mathbin{|}{}}#2\}}
\def\tensor{\otimes}
\newtheorem{theorem}{Theorem}
\newtheorem{lemma}[theorem]{Lemma}
\theoremstyle{remark}

\begin{document}

\ArxivOrBLMS{%
\title[Root systems for Lorentzian Kac-Moody
  Algebras (arXiv version)]{Root systems for Lorentzian Kac-Moody Algebras in
  rank~$3$ (arXiv version)}%
}{%
\title{Root systems for Lorentzian Kac-Moody Algebras in rank~$3$}%
}%

\author{Daniel Allcock}
\address{Department of Mathematics, University of Texas at Austin,
  Austin, TX}
\email{allcock@math.utexas.edu}
\thanks{Supported by NSF grant DMS-1101566}



\date{28 October 2014}

\begin{abstract}
Sometimes a hyperbolic Kac-Moody algebra admits an automorphic
correction, meaning a generalized Kac-Moody algebra with the same real
simple roots and whose denominator function has good automorphic
properties.  These for example allow one to work out the root
multiplicities.  Gritsenko and Nikulin have formalized this in their
theory of Lorentzian Lie algebras and shown that the real simple roots
must satisfy certain conditions in order for the algebra to admit an
automorphic correction.  We classify the hyperbolic root systems of
rank~$3$ that satisfy their conditions and have only finite many
simple roots, or equivalently a timelike Weyl vector.  There are
$\numsystems$ of them, with as many as $24$ simple roots.  Patterns in
the data suggest that some of the non-obvious cases may be the
richest.
\end{abstract}

\maketitle

\ArxivOrBLMS{%
\noindent{\bf NOTE:} This is an unabridged version of the paper 
``Root systems for Lorentzian Kac-Moody algebras in rank~$3$'', to
appear in the {\it Bulletin of the London Math. Society.}  The editors
asked me to post the full table of {\numsystems} root systems here, while
printing only the {\numwithfoursimpleroots} cases with $4$ simple roots in the {\it
  Bulletin.}  Besides including the full table, we added in
section~\ref{sec-how-to-read}
a second 
example of how to read the table, and there are various
minor wording differences.  Also, the full table is available in
computer readable form, commented out in this version's
{\TeX} source.
\bigskip
}{}

\noindent
Kac and Moody introduced the Lie algebras that now bear their names in
\cite{Kac-first} and \cite{Moody-first}.  The main goal at the time
was understanding the affine algebras---the central extensions of the
algebras of loops in finite dimensional Lie algebras.  But it was
clear that the construction also yields hyperbolic algebras (among
others).  Here ``affine'' and ``hyperbolic'' refer to the action of
the Weyl group $W$ on Euclidean and hyperbolic space, and the elliptic
case is the classical one of finite $W$.  The hyperbolic algebras are
the next step beyond the affine ones, but not much is known about them:
even finding root multiplicities is very hard.
One difficulty is that only the
best hyperbolic KMAs are likely to be interesting and it is not clear
which ones are best.  Various authors have tried to identify the gems,
and our paper continues this process.

Borcherds has proposed the following criterion for a Kac-Moody algebra
to be interesting: one must be able to know both the simple roots and
the multiplicity of an arbitrary root (in some more concrete manner
than the Peterson recursion formula).  No known hyperbolic examples
meet this criterion.  But after a careful analysis of the case when
the simple roots are those of the 26-dimensional even unimodular
Lorentzian lattice, Borcherds discovered that his criterion could be
met by enlarging the KMA slightly.  This enlargement was not a KMA but
something he called a generalized KMA, got by adjoining some imaginary
simple roots \cite{Borcherds-fake-monster-Lie-algebra}.  This
particular GKMA is now called the fake monster Lie algebra, and the
same method allowed him to construct and analyze the monster Lie
algebra, leading to his famous proof of the Conway-Norton moonshine
conjectures \cite{Borcherds-moonshine}.  The key to his analyses was
that the denominator functions of these GKMAs happen to be automorphic
forms.  Gritsenko and Nikulin developed these examples into their
theory of Lorentzian Lie algebras, with the automorphic property at
its heart.
The essential axiom is that the
denominator function be a reflective automorphic form.
See \cite{Nikulin-A-theory-of-Lorentzian-KMAs} and
\cite{GN-On-classification-of-LKMAs} for further information.  

They show that this can only happen if the set $\Pi$ of real simple
roots has certain properties.  The $\Pi$ with these properties fall
into two classes, the ``elliptic'' ones with finitely many simple
roots and the ``parabolic'' ones with infinitely many.  We completely
classify the elliptic ones in rank~$3$: there are $\numsystems$ simple
root systems, with as many as 24 
simple roots.  The parabolic case is also very interesting and one
could probably classify them.  But one would have to keep in mind
Nikulin's construction of infinitely many of them by modifying a
single example
\cite[example~1.3.4]{Nikulin-Reflection-groups-in-hyperbolic-spaces-and-the-denominator-formula-for-LKMLAs}.

Complementing these elliptic and parabolic classes are the
``hyperbolically reflective'' hyperbolic root systems considered in
\cite{Nikulin-rank-3}.  These do not correspond to Lorentzian Lie
algebras because conditions (i)--(iii) below rule out the possibility
of a spacelike Weyl vector ($\rho^2>0$), whereas only a ``generalized
Weyl vector'' was assumed in \cite{Nikulin-rank-3}.  But they are
still interesting from the perspective of reflection groups and
automorphic forms.

Carbone et.\ al., building on earlier work, published extensive
tables of hyperbolic Dynkin diagrams and associated data \cite{Carbone-etal}.  They address a
different but related problem.  They start with the idea that the
Dynkin diagram for a root system in $n$-dimensional Minkowski space
should have $n+1$ nodes (and finite volume Weyl chamber).  The
assumption on the number of nodes is equivalent to the chamber being a
simplex, and we think this restriction is not very natural.  It is
carried over from the classical case of finite Weyl group, where it
just so happens that the chamber is always a spherical simplex.  In
hyperbolic space this isn't true, and the simplicial chambers aren't
particularly special among all chambers.  

On the other hand, Carbone et.\ al.\ treat the non-symmetrizable case
too.  Our formulation in terms of lattices restricts us to the
symmetrizable case.  Of the 123 rank~$3$ root systems in
\cite{Carbone-etal}, $44$ of them are symmetrizable.  These are the
$42$ in tables $1$ and~$2$ of Sa\c{c}lio\u{g}lu's \cite{Saclioglu},
plus two that Sa\c{c}lio\u{g}lu missed.  We checked that
{\numwiththreesimpleroots} systems with three simple roots correspond
bijectively with this list.  To save space we haven't printed Dynkin
diagrams for our root systems, but these can be read from what we have
printed; see section~\ref{sec-how-to-read}.

To state our main result
precisely we make the following mostly standard definitions.

\medskip
A {\it lattice}
$L$ means a finitely generated free abelian group equipped with a
symmetric $\Q$-valued bilinear pairing.  We denote the pairing
of $x,y\in L$ by $x\cdot y$ and write $x^2$ for the norm $x\cdot x$ of
$x$.  We call $L$ integral if all inner products lie in $\Z$, and call
$L$ unscaled if it is integral and the gcd of all inner products
is~$1$.  

We call $\a\in L$ a \defn{root} of $L$ if $\a$ has positive norm and
$L\cdot \a\sset\frac{1}{2}\a^2\Z$.  In this case the reflection in
$\a$, negating $\a$ and fixing $\a^\perp$ pointwise, preserves $L$.  A
root need not be primitive in $L$.  The reason for allowing
imprimitive roots is that the ``root lattice'' in the construction of the
associated KMA \cite{Kac-book} is the free abelian group $\Z^\Pi$ on
$\Pi$.  Our $L$ is merely a quotient of $\Z^\Pi$.  Although roots are
obviously primitive in $\Z^\Pi$, there is no reason to expect their
images in $L$ to also be primitive.  And there are known examples
with imprimitive roots, like the 32nd entry in
\cite[table~1]{GN-AFaLKMAsI}, which corresponds to our $\Pi_{8,11}$.

We call a set $\Pi$ of spacelike (positive-norm) vectors in Minkowski
space $\R^{n,1}$ a \defn{simple root
  system} if $\a\cdot\a'$ is nonpositive and lies in
$\frac{1}{2}\a^2\Z$, for all $\a,\a'\in\Pi$. 
Then they are roots of their integral span, called the root
lattice $L$.  The \defn{Weyl group} $W$ means the group generated
by the reflections in the roots.  To avoid degenerate cases we assume that
$L$ has full rank in $\R^{n,1}$, which implies $|W|=\infty$.  
In this case the \defn{fundamental chamber}
\begin{equation}
\label{eq-fundamental-chamber-for-Tits-cone}
C:=\set{x\in L\tensor\R}{\hbox{$x\cdot\a\leq0$ for all $\a\in\Pi$}}
\end{equation}
meets just one of the two cones of negative norm vectors, which we
call the \defn{future cone}.  Furthermore, $C$ is a
fundamental domain for the action of $W$ on the \defn{Tits cone}
$T:=\cup_{w\in W}\,wC$ in the sense that every point of $T$ is
$W$-equivalent to a unique point of $C$.  The Tits cone always
contains the future cone.

Now we can state the conditions of Gritsenko and Nikulin on the set
$\Pi$ of real simple roots of a GKMA, that are necessary for its
denominator function to be an automorphic form.  First, $\Pi$ must be
a simple root system spanning $\R^{n,1}$ for some~$n$.  Furthermore,

(i) the interior $T^\circ$ of the Tits cone must coincide with the
future cone;

(ii) the normalizer of $W$ must have finite index
in $\orthogonalgroup(L)$;

(iii) a Weyl vector must exist, meaning a vector $\rho\in\R^{n,1}$ with
$\rho\cdot\a=-\a^2/2$ for all $\a\in\Pi$.

In the language of \cite{Nikulin-Reflection-groups-in-hyperbolic-spaces-and-the-denominator-formula-for-LKMLAs},
the first condition is that $\Pi$ has arithmetic type and the
second refines this to restricted arithmetic type; also we say
``Weyl vector'' in place of
\cite{Nikulin-Reflection-groups-in-hyperbolic-spaces-and-the-denominator-formula-for-LKMLAs}'s
``lattice Weyl vector''.   If it exists then $\rho$ is unique and lies in~$C$.

Under these conditions, it turns out that $\rho$ must be timelike
($\rho^2<0$) or lightlike ($\rho^2=0$), in which cases $\Pi$ is said
to have elliptic or parabolic type.  In the timelike case the
Weyl chamber must be a finite-volume polytope in hyperbolic space (the
projectivized future cone).
We are abusing language slightly here: we will say that $C$ has finite volume if
its image in hyperbolic space does, and similarly for compactness.
In the lightlike case the Weyl chamber has infinite volume and
infinitely many sides.

\begin{theorem}
\label{thm-main}
Suppose $\Pi$ is a simple root system spanning $\R^{2,1}$ and
satisfies (i)--(iii) with timelike Weyl vector $\rho$.  
Then up to
scale, $\Pi$ is isometric to exactly one of $\numsystems$ simple
root systems. 
\end{theorem} 

\ArxivOrBLMS{%
The {\numsystems} simple root systems  appear at the end of the paper.
}{%
The {\numwithfoursimpleroots} cases 
with four simple roots appear at the end of this
paper and the rest appear in the unabridged online 
version \cite{Allcock-arxiv}.
}%
The number of simple root systems with $n$ simple roots is:
$$
\setcounter{MaxMatrixCols}{20}
\begin{matrix}
n=
&3
&4
&5
&6
&7
&8
&9
&10
&11
&12
&13
\\
&\numwiththreesimpleroots
&\numwithfoursimpleroots
&\numwithfivesimpleroots
&\numwithsixsimpleroots
&\numwithsevensimpleroots
&\numwitheightsimpleroots
&\numwithninesimpleroots
&\numwithtensimpleroots
&\numwithelevensimpleroots
&\numwithtwelvesimpleroots
&\numwiththirteedsimpleroots
\\
\phantom{0}
\\
n=
&14
&15
&16
&17
&18
&19
&20
&21
&22
&23
&24
\\
&\numwithfourteensimpleroots
&\numwithfifteensimpleroots
&\numwithsixteensimpleroots
&\numwithseventeensimpleroots
&\numwitheightteensimpleroots
&\numwithnineteensimpleroots
&\numwithtwentysimpleroots
&\numwithtwentyonesimpleroots
&\numwithtwentytwosimpleroots
&\numwithtwentythreesimpleroots
&\numwithtwentyfoursimpleroots
\end{matrix}
$$

We remark that for fixed $n$ there are only finitely many simple root
systems spanning $\R^{n,1}$ and satisfying (i)--(iii) with timelike
$\rho$, by
\cite[thm.~1.3.2]{Nikulin-Reflection-groups-in-hyperbolic-spaces-and-the-denominator-formula-for-LKMLAs},
and $n\leq22$ by work of Esselmann \cite{Esselmann}.  For $n\leq19$,
there are infinitely many simple root systems in $\R^{n,1}$ with
finite-volume chamber \cite{Allcock-19}, and their number even grows
exponentially with this volume (except perhaps for $n=16,17$).  The
essential condition that narrows this multitude down to finitely many
is the existence of~$\rho$.  In the lightlike case Nikulin has proven
a finiteness result
\cite[thm.~1.3.3]{Nikulin-Reflection-groups-in-hyperbolic-spaces-and-the-denominator-formula-for-LKMLAs}
for fixed $n$, and also the bound $n\leq998$
\cite{Nikulin-Algebraic-surfaces-with-log-terminal-singularities}.

These necessary conditions (i)--(iii) for the denominator function
to be ``correctable'' to a reflective
automorphic form have also proven sufficient in every case
investigated so far.  That is, given such a $\Pi$, it has proven
possible to ``automorphically correct'' \cite[(3.3)]{GN-Igusa} the KMA
with simple root system $\Pi$ to a GKMA $\g$ (possibly a Lie superalgebra rather
than a Lie algebra) with
the same real simple roots but also some imaginary simple roots, such that
$\g$'s 
denominator function is a reflective automorphic form.  That this should always
be possible is hinted at by Borcherds' suggestion
\cite[\S12]{Borcherds-automorphic-forms-with-singularities-on-Grassmannians}
that ``good'' reflection groups and ``good'' automorphic forms should
correspond to each other, and formalized by the arithmetic
mirror-symmetry conjecture of Gritsenko and Nikulin
\cite[conj.~2.2.4]{GN-AFaLKMAsI}.  Specific examples of automorphic
corrections are worked out in
\cite{GN-Igusa},
\cite{GN-Siegel-Automorphic-form-corrections-of-some-LKMLAs},
\cite{GN-K3-surfaces-Lorentzian-KMAs-and-Mirror-Symmetry}
and \cite{GN-AFaLKMAsII}; see
section~\ref{sec-Gritsenko-Nikulin} for more details of the
correspondences between those simple root systems and ours.  Gritsenko's
paper \cite{Gritsenko-Reflective-modular-forms-in-algebraic-geometry}
promises further automorphic corrections in a forthcoming paper with
Nikulin.

We remark briefly on Lorentzian Lie superalgebras.  Our conditions
(i)--(iii) on the simple root system are still necessary, and there are two
additional conditions.  First, a real root must satisfy
$L\cdot\a\sset\a^2\Z$ before its parity can be set to odd.  Second,
there is a sign condition on the coefficients of the Fourier-Weyl
expansion of the reflective modular form.  The current paper doesn't
address the sign condition.  But a complete list of candidates, for
simple root systems of rank~$3$ Lorentzian Lie superalgebras with timelike
Weyl vector, can be got from our list by examining our simple root systems
and checking which roots satisfy $L\cdot\a\sset\a^2\Z$.

\smallskip
The obvious open
problem is to investigate the possible automorphic corrections for
these $\numsystems$ simple root systems.  For a given simple root system $\Pi$ this
amounts to seeking reflective automorphic forms for subgroups $\Gamma$
of $\orthogonalgroup(M)$, where $M$ is a lattice of signature $(2,3)$ that possesses an
isotropic vector  orthogonal to a copy of
$\Pi$. ($\Pi$ should consist of roots of $M$ and $\Gamma$ should contain $W$.)  See
\cite{GN-AFaLKMAsII} for the largest investigation of this sort.

Surely only some of our root systems are the gems we seek, and our
results suggest that certain cases are worth investigating first.  One
very surprising outcome is that $\latticeHQcount$ of the $\numsystems$
cases share a single root lattice $\latticeHQ\iso\spanof{-1}\oplus
A_2[15]$.  (See section~\ref{sec-how-to-read} for lattice names.)  This suggests that
$\orthogonalgroup\bigl(\latticeHQ\oplus\hypcell{1}\bigr)$ and its finite-index
subgroups may admit an extraordinarily rich family of reflective
automorphic forms.  Or perhaps there is a single automorphic form
for some many-cusped subgroup, whose Fourier
expansions at various cusps give automorphic corrections to the
KMAs associated to many of the
root systems with root lattice $\latticeHQ$?  A few other lattices
account also occur very frequently, but none of them are ``obvious''
lattices.  
For example, $\latticeHO$ is an index
$2$ superlattice of $\spanof{-4,36,36}$ and occurs $\latticeHOcount$
times.  This is the only one of the $10$ most frequently occurring
root lattices
which is isotropic.  Interestingly, $\orthogonalgroup(\latticeHO)$ is a subgroup
of $\orthogonalgroup(\Z^{2,1})=\PGL_2\Z\times\{\pm1\}$, whose root system defines
the first- and most-studied hyperbolic KMA, introduced by
Feingold-Frenkel in \cite{Feingold-Frenkel}.  See section~\ref{sec-how-to-read} for
more about the frequently-occurring root lattices.

\smallskip
In section~\ref{sec-proof} we prove the main theorem, and in section~\ref{sec-how-to-read} we
explain how to read the table.  We have presented each simple root system so
that its essential properties are visible, including its symmetry
group, root norms, Weyl vector, the angles of its Weyl chamber, and
the isomorphism type of the root lattice.  We also provide additional
information like largest and smallest Weyl chambers and the number of
chambers with interesting properties like compactness or
right-angledness.  Finally, section~\ref{sec-Gritsenko-Nikulin} is devoted to correlating
our classification with the extensive work of
Gritsenko and Nikulin, and contains further references to the literature.
This includes proving
conjecture~1.2.2 of
their paper \cite{GN-AFaLKMAsI}.
\ArxivOrBLMS{%
The table itself appears at the very end of the paper.
All information in the table is available in computer-readable form,  commented out
in the {\TeX} file.
}{%
The table of simple root systems with four roots appears at the very
end of the paper.  For simple root systems with $3,5,6,7\dots,24$ roots, see the online version
of this paper \cite{Allcock-arxiv}, which also contains all displayed information in computer-readable form, commented out
in the {\TeX} file.
}

\section{Proof of the classification}
\label{sec-proof}

We begin by supposing that $\Pi$ is a simple root system spanning
$\R^{2,1}$, with finite-area Weyl chamber.
We write $L$ for the
root lattice.  The chamber has finitely many sides (since it has
finite area), so it is a polygon.  The roots correspond to these
sides, so they have a natural cyclic ordering, up to reversal.  We
suppose $\Pi=\{\a_1,\dots,\a_n\}$ with the $\a_i$ ordered in this
manner.

\begin{lemma}
\label{lem-narrow-matrices}
Some $3$, $4$ or~$5$ consecutive members of $\Pi$ have inner product
matrix equal to one of those given below in
\eqref{eq-shape-of-2d-SE-integral-polygon-orthogonality-case}--\eqref{eq-shape-of-2d-CP-integral-polygon},
up to multiplication by a rational number.

In each case $A,B,A',B',C,C'$ are positive integers, except in cases
\eqref{eq-shape-of-2d-SE-integral-polygon-orthogonality-case} and \eqref{eq-shape-of-2d-SP-integral-polygon-orthogonality-case} where $A'$ and $B'$ are both taken to be $0$.
In cases \eqref{eq-shape-of-2d-SP-integral-polygon-orthogonality-case}--\eqref{eq-shape-of-2d-CP-integral-polygon}, $k$ is a positive integer, and in case
\eqref{eq-shape-of-2d-CP-integral-polygon} so is~$k'$.
In every case we require
\begin{equation*}
AB'C' = A'BC
\end{equation*}
and write $\beta$ for the common value.
In every case we also require $CC'<4K^2$ where
\begin{equation*}
K:=1+\frac{\sqrt{AB}}{2}+\frac{\sqrt{A'B'}}{2}
+\sqrt{\bigl(2+\sqrt{AB}\,\bigr)\bigl(2+\sqrt{A'B'}\,\bigr)}
\end{equation*}
In cases \eqref{eq-shape-of-2d-SP-integral-polygon-orthogonality-case}--\eqref{eq-shape-of-2d-CP-integral-polygon}, we make the further requirement on
$A,B,A',B',C,C'$ that
\begin{equation*}
N:=4+4\frac{CC'+\beta+A'B'}{AB-4}
\end{equation*}
is an integer, and require $k$ to divide it.  In case \eqref{eq-shape-of-2d-CP-integral-polygon} we define
$N'$ analogously with primed and unprimed letters exchanged, and
require that it is an integer and that $k'$ divides it.  Finally, in
case \eqref{eq-shape-of-2d-CP-integral-polygon} we also
require that 
${\gamma k^2}/{A'BN}$
and
${\gamma k'^{\,2}}/{AB'N'}$
are integers,
where
\begin{equation*}
\gamma:=
\frac{2\beta}{kk'}
\Biggl(2+\frac{\beta}{CC'}-
\frac{(2CC'+\beta)^3}{(AB-4)(A'B'-4)C^2C'^{\,2}}
\Biggr).
\end{equation*}
The inner product matrices are:
\begin{equation}
\label{eq-shape-of-2d-SE-integral-polygon-orthogonality-case}
\begin{pmatrix}
2AC &   -ABC   &   -ACC'   \\
-ABC    & 2BC  &     0     \\
-ACC'   &     0    & 2AC'
\end{pmatrix}
\qquad
AB\leq4;
\end{equation}

\begin{equation}
\label{eq-shape-of-2d-SE-integral-polygon-NONorthogonality-case}
\begin{pmatrix}
2AB'&   -ABB'   &    -\beta   \\
-ABB'   & 2BB'  &  -A'B'B  \\
-\beta     &  -A'B'B   & 2A'B
\end{pmatrix}
\qquad
\begin{matrix}
AB\leq4\\
A'B'\leq4;
\end{matrix}
\end{equation}

\begin{equation}
\label{eq-shape-of-2d-SP-integral-polygon-orthogonality-case}
\begin{pmatrix}
2 AC &   0             &  -ABC   &   -ACC' \\
0       & 2 AC'N/k^2 &    0    & -AC'N/k \\
-ABC    &   0             & 2 BC &     0   \\
-ACC'   & -AC'N/k         &    0    & 2'AC'
\end{pmatrix}
\qquad
\begin{matrix}
4<AB<36\\
4<CC';
\end{matrix}
\end{equation}

\begin{equation}
\label{eq-shape-of-2d-SP-integral-polygon-NONorthogonality-case}
\begin{pmatrix}
2AB'&        0      &    -ABB'    &     -\beta         \\
0       &2A'BN/k^2&      0      &   -A'BN/k   \\
-ABB'   &        0      & 2 BB'   &    -A'B'B       \\
-\beta     &  -A'BN/k      &    -A'B'B   &   2A'B
\end{pmatrix}
\qquad
\begin{matrix}
4<AB<36\\
A'B'\leq4\\
4<CC';
\end{matrix}
\end{equation}

\begin{gather}
\label{eq-shape-of-2d-CP-integral-polygon}
\begin{pmatrix}
2 AB' &       0         &  -ABB'   &  -AB'N'/k'        &     -\beta    \\
     0   & 2 A'BN/k^2 &    0     &  \gamma          & -A'BN/k \\
   -ABB' &       0         & 2 BB' &        0          &    -A'B'B   \\
-AB'N'/k' &      \gamma   &    0     & 2AB'N'/k'^{\,2} &       0   \\
   -\beta   &  -A'BN/k        & -A'B'B   &        0          & 2A'B
\end{pmatrix}
\\
\notag
4<AB<36
\qquad
4<A'B'<36
\qquad
4<CC'.
\end{gather}
\end{lemma}

\begin{proof}
All the hard work was done in \cite{Allcock-rk3}, which used a
refinement of
Nikulin's method of narrow parts of polygons.  First, theorem~1 of
that paper uses the finite-sidedness of the chamber to show that the
chamber has one of three features: a ``short edge'' orthogonal to at
most one of its neighbors (which are not orthogonal to each other), or
at least $5$ edges and a ``short pair'', or at least $6$ edges and a
``close pair''.  Second, lemmas 7--9 of the same paper show that in
each of these cases there exist 3, 4 or 5 consecutive roots with
one of the inner product matrices \eqref{eq-shape-of-2d-SE-integral-polygon-orthogonality-case}--\eqref{eq-shape-of-2d-CP-integral-polygon}.  We remark that the
notion of root used in this paper is the same as that of a weight~$2$
quasiroot in \cite{Allcock-rk3}, except that \cite{Allcock-rk3} assumes
all quasiroots are primitive.  Since primitivity played no role in the
proofs of these lemmas, they apply in the current situation.
\end{proof}

\begin{proof}[Proof of theorem~\ref{thm-main}]
We first explain how we assembled the list, using a
refinement of the proof of Gritsenko-Nikulin's Thm.~1.2.1 of \cite{GN-AFaLKMAsI}.
We began by constructing all
the matrices from lemma~\ref{lem-narrow-matrices}, using a computer.  There are
$\nummatrices$ of them, the same set of matrices we started with in
the proof of \cite[thm.~13]{Allcock-rk3}.  Then we discarded the ones that
are positive definite ($\numposdefinite$ of them) or have rank~$2$
($\numranktwo$ of them).  For each remaining matrix $M$, we considered
the lattice $D$ got as the quotient of $\Z^{k=3,4\ {\rm or}\ 5}$ by
the kernel of $M$, and the vectors $\a_1,\dots,\a_k$ which are the
images of the standard basis vectors for $\Z^k$.  By the construction
of the matrices \eqref{eq-shape-of-2d-SE-integral-polygon-orthogonality-case}--\eqref{eq-shape-of-2d-CP-integral-polygon},
$\a_1,\dots,\a_k$ is a simple root system in $D$.  In each
case we sought a Weyl vector.  This was easy: since
$k\geq3$, the conditions $\rho\cdot\a_i=-\a_i^2/2$ determine (or
overdetermine) $\rho$.  We discarded $(D,\a_1,\dots,\a_k)$ if no Weyl
vector existed, or if a Weyl vector existed but was not timelike.  This accounted for $\numDwithnoWeylvector$ of the  cases, leaving
$\numDwithWeylvector$.

For each remaining tuple $(D,\a_1,\dots,\a_k)$ we computed the
reflective hull
of
$\a_1,\dots,\a_k$.  This means the set of vectors $x\in D\tensor\Q$
satisfying $x\cdot\a_i\in\frac{1}{2}\a_i^2\Z$ for all $i=1,\dots,k$.
It is the unique maximal-under-inclusion lattice in which
$\a_1,\dots,\a_k$ are roots.  Because $\a_1,\dots,\a_k$ span
$D\tensor\Q$, the hull contains $D$ of finite index.  We enumerated
all lattices $E$ lying between $D$ and the hull.  There were a total
of $\numenlargements$ of these enlargements.  We rescaled each to be
unscaled.

The following definition is useful for further analysis; it is a
codification of the properties possessed by our $\numenlargements$
tuples $(E,\a_1,\dots,\a_k)$.  If $E$ is an integral lattice of
signature $(2,1)$ then a {\it chain\/} in $E$ means a sequence
$\a_1,\dots,\a_{m\geq3}$ of roots of $E$, which form a simple root system, admit
a timelike Weyl vector $\rho$, and have the property that each span
$\spanof{\a_1,\a_2},\dots,\spanof{\a_{m-1},\a_m}$ is positive-definite
or positive-semidefinite.  The Weyl vector is uniquely determined
since $m\geq3$.  The geometric meaning of the last condition is that
the mirrors of any two consecutive roots meet in $H^2$ or its
boundary.  We say that the chain is {\it closed\/} if
$\spanof{\a_m,\a_1}$ satisfies the same condition.  An {\it
  extension\/} of the chain means a root $\a_{m+1}$ of $E$ such that
$\a_1,\dots,\a_{m+1}$ is also a chain in $E$.

Now we address the question: what are all the possible extensions of a
non-closed chain $\a_1,\dots,\a_m$ in an integral lattice $E$?  In
particular, are there any?  Our method is a refinement of one used
by Gritsenko and Nikulin \cite[thm.~1.2.1]{GN-AFaLKMAsI}.  One can restrict the norm of a root
$\a$ of $E$, using the fact that one of the following holds: (i) $\a$
is imprimitive, in which case $\a/2$ spans an orthogonal summand of
$E$; (ii) $\a$ is primitive and spans an orthogonal summand of $E$; or
(iii) $\a$ is primitive and spans an orthogonal summand of an
index~$2$ sublattice of $E$.  In all three cases it follows that $\a^2$ divides $4e$,
where $e$ is the largest elementary divisor of the discriminant group
$E^*/E$ of $E$.  

So suppose $N$ a positive divisor of $4e$.  If an
extension $\a_{m+1}$ has norm~$N$, then its inner product $I$ with
$\a_m$ must be one of a few possibilities.  This uses the fact that
$\a_m$ and $\a_{m+1}$ are simple roots for a finite or affine Coxeter
group.  For example, if $\a_m^2=\a_{m+1}^2=N$ then $\a_m$ and
$\a_{m+1}$ must be simple roots for an $A_1^2$, $A_2$ or $\Atilde_1$.
In these cases, $I=\a_m\cdot\a_{m+1}$ would be $0$, $-N/2$ or $-N$
respectively.  The full rules for determining the
possibilities for $I$ are:
$$
\def\OR{\hbox{ or }}
\begin{matrix}
N/\a_m^2&\hbox{possible root systems}&\hbox{possibilities for $I$}
\\
4&\Atilde_1'\OR A_1^2&-N/2\OR0
\\
3&G_2\OR A_1^2&-N/2\OR0
\\
2&B_2\OR A_1^2&-N/2\OR0
\\
1&\Atilde_1\OR A_2\OR A_1^2&-N\OR-N/2\OR0
\\
1/2&B_2\OR A_1^2&-N\OR0
\\
1/3&G_2\OR A_1^2&-3N/2\OR0
\\
1/4&\Atilde_1'\OR A_1^2&-2N\OR0
\\
\hbox{other}&A_1^2&0
\end{matrix}
$$
Here $\Atilde_1$ indicates two independent vectors of the same norm,
whose sum is isotropic.  And $\Atilde_1'$ indicates two independent
vectors, one having four times the norm of the other, such that the
longer one plus twice the shorter one is isotropic.

After fixing such a pair $(N,I)$, 
one
can seek extensions $\a_{m+1}$ with
$\a_{m+1}^2=N$ and  $\a_m\cdot\a_{m+1}=I$.  
We record any that exist, and continue to the next pair $(N,I)$.
Here are the details.
The conditions
$\a_{m+1}\cdot\rho=-N/2$ and $\a_{m+1}\cdot\a_m=I$ restrict $\a_{m+1}$
to a line in $\R^{2,1}$, and imposing the condition $\a_{m+1}^2=N$
restricts $\a_{m+1}$ to one or two possibilities.  Finding them amounts to
solving a quadratic equation.  If it gives non-rational
vectors (the most common result) then
the sought extension doesn't exist.  If the solution or
solutions are rational then one checks which ones have nonpositive inner
products with $\a_1,\dots,\a_m$.  One can show (or just check) that at
most one does.  If none do then the sought extension
doesn't exist.  So
suppose exactly one does.  
If it is a root of $E$ then it is the
extension sought; otherwise the sought extension doesn't exist.

We began with our $\numenlargements$ chains, checked which chains were
closed ($\numclosedstepone$ of them), and found all the extensions
($\numextensionsstepone$ of them) of the non-closed chains.  We set
the closed ones aside for later analysis and repeated the process with
the $\numextensionsstepone$ new chains, finding $\numclosedsteptwo$
closed chains and $\numextensionssteptwo$ extensions.  Continuing in
this manner, on the $21$st iteration we found $\numclosedlaststep$
closed chains and no extensions.
This left us with $\numclosedchains$ tuples $(E,\a_1,\dots,\a_m)$ with
$E$ an integral lattice and $\a_1,\dots,\a_m$ a closed chain in $E$.
The main content of the theorem is that $(L,\Pi)$ is isometric to one of
them, up to dihedral rearrangement of the simple roots.  

Given our
preparations, this is easy: by lemma~\ref{lem-narrow-matrices}, some $3$, $4$ or $5$
consecutive members of $\Pi$ have one of the inner product matrices
given there, up to scale.  Call them $\a_1,\dots,\a_k$.  Since $L$
lies in the reflective hull of $\a_1,\dots,\a_k$, the tuple
$(L,\a_1,\dots,\a_k)$ occurs on the list of $\numenlargements$
chains.  Obviously each of $\a_{k+1},\dots,\a_n$ is an extension of
the chain of its predecessors, and $\a_1,\dots,\a_n$ is closed.  So
$(L,\a_1,\dots,\a_n)$ must occur on our final list of
$\numclosedchains$ closed chains.

We have established the completeness of the enumeration, but more
work was required to weed out redundancy.  
One source of redundancy is
that in one of our
$\numclosedchains$ tuples $(E,\a_1,\dots,\a_n)$, $E$ may be strictly
larger than $\spanof{\a_1,\dots,\a_n}$.  
These cases should be discarded because by definition $L$ is span of
the simple roots.  
We eliminated this redundancy by
computing the inner product matrix of the roots, which effectively
forgets $E$, and then rescaling to make the root lattice unscaled.
A second source of redundancy  is 
dihedral reorderings of roots.  
We eliminated this in the obvious way:
whenever several of these inner product matrices were equal up to
dihedral permutation, we kept only one.  The result was our list of
$\numsystems$ simple root systems.
\end{proof}

\section{How to read the table}
\label{sec-how-to-read}

The simple root systems are named $\Pi_{i,j}$ where $i$ indicates the
number of roots.  For given $i$, the simple root systems are organized
according to their symmetry groups, but those with a given symmetry
group appear in no special order.  For each $\Pi$ we have in mind its
simple roots $\a_1,\dots,\a_n$ in cyclic order, although not all of
them are listed when symmetries are present.  
\ArxivOrBLMS{%
We will give two
examples illustrating how to read table entries, and then give the
full rules.
}{%
\par
We will give an
example illustrating how to read table entries, and then give the
full rules.  Some of the general information 
given below refers to the full table, rather than just the portion we
have printed.  See \cite{Allcock-arxiv} for the rest of the table.
We have also omitted a few table-reading rules that are relevant 
only to the unprinted portion;  see \cite{Allcock-arxiv} for details
}

\newbox\namebox
\newbox\shapebox
\newbox\matricesbox
\newdimen\halfwidth
\halfwidth=\hsize
\divide\halfwidth by2
\newdimen\myboxwidth
\newbox\onelinebox

\smallskip
{\it Extended example:} 
Consider the entry for $\Pi_{4,5}$, which reads
\begin{center}
\setbox\matricesbox=\hbox{%
{$\left[\!\llap{\phantom{%
\begingroup \smaller\smaller\smaller\begin{tabular}{@{}c@{}}%
\phantom{0}\\\phantom{0}\\\phantom{0}
\end{tabular}\endgroup%
}}\right.$}%
\begingroup \smaller\smaller\smaller\begin{tabular}{@{}c@{}}%
-1/4\\\phantom{0}\\\phantom{0}
\end{tabular}\endgroup%
\kern3pt%
\begingroup \smaller\smaller\smaller\begin{tabular}{@{}c@{}}%
\phantom{0}\\3\\\phantom{0}
\end{tabular}\endgroup%
\kern3pt%
\begingroup \smaller\smaller\smaller\begin{tabular}{@{}c@{}}%
\phantom{0}\\\phantom{0}\\15
\end{tabular}\endgroup%
{$\left.\llap{\phantom{%
\begingroup \smaller\smaller\smaller\begin{tabular}{@{}c@{}}%
\phantom{0}\\\phantom{0}\\\phantom{0}
\end{tabular}\endgroup%
}}\!\right]$}%
{$\left[\!\llap{\phantom{%
\begingroup \smaller\smaller\smaller\begin{tabular}{@{}c@{}}%
0\\0\\0
\end{tabular}\endgroup%
}}\right.$}%
\begingroup \smaller\smaller\smaller\begin{tabular}{@{}c@{}}%
2\\1\\0
\end{tabular}\endgroup%
\kern3pt%
\begingroup \smaller\smaller\smaller\begin{tabular}{@{}c@{}}%
6\\0\\1
\end{tabular}\endgroup%
{$\left.\llap{\phantom{%
\begingroup \smaller\smaller\smaller\begin{tabular}{@{}c@{}}%
0\\0\\0
\end{tabular}\endgroup%
}}\!\right]$}%
}%
\ifdim\wd\matricesbox>\halfwidth\myboxwidth=\hsize\else\myboxwidth=\halfwidth\fi
\leavevmode
\vbox{%
\ifdim\myboxwidth=\hsize
\setbox\onelinebox=\hbox{%
\vbox{\hbox{%
$\Pi_{4,5}$ spans $L_{16.8}$%
}\hbox{%
$|6|6|6|6\rtimes D_{4}$%
}%
}%
\hfill\copy\matricesbox
}%
\ifdim\wd\onelinebox>\myboxwidth
\hbox to \myboxwidth{%
$\Pi_{4,5}$ spans $L_{16.8}$%
\hfil
$|6|6|6|6\rtimes D_{4}$%
}%
\box\matricesbox
\else
\hbox to \myboxwidth{%
\unhbox\onelinebox
}%
\fi
\else
\hbox{%
$\Pi_{4,5}$ spans $L_{16.8}$%
\hfil}%
\hbox{%
$|6|6|6|6\rtimes D_{4}$%
\hfil}%
\box\matricesbox
\fi
}%
\end{center}
The subscripts on $\Pi$ mean that this is the 5th on the list of
simple root
systems with $4$ simple roots.  The first thing to look at is
``$|6|6|6|6$'', which tells us that the angles of the Weyl chamber are
$\pi/6$, $\pi/6$, $\pi/6$ and $\pi/6$.  Ignore the vertical bars for
now.  The first matrix specifies an inner product on $\Q^3$.  The columns of the
second matrix are (some of) the simple roots.  In this example these are
the vectors $(2,1,0)$ and $(6,0,1)$.  The coordinates have been chosen
so that the norm of every root is equal to its first coordinate, so
these roots have norms $2$ and~$6$.  (In particular, by scanning the
top line of the matrix, one can read off the norms of the simple
roots.)  

The reason we have listed only some of the roots is that the rest are
obtained from the symmetry group of this simple root system, which is
the dihedral group of order~$4$ indicated by ``$\semidirect D_4$''.
In every case with this symmetry group, we have chosen coordinates so
that $D_4\iso (\Z/2)^2$ acts by changing the signs
of the last two coordinates.  The full set of simple roots is obtained
by applying these transformations to the simple roots that were given
explicitly, yielding $(2,\pm1,0)$ and $(6,0,\pm1)$.  Now we can
explain the vertical bars in ``$|6|6|6|6$''.  These are only present
when the symmetry group contains reflections, and then they indicate
how their mirrors meet the boundary of the chamber.  Remember that the
$6$'s correspond to the vertices of chamber.  A bar between two of
them indicates that the edge between the corresponding vertices is
bisected by one of the mirrors of the symmetry group.  The notation
should be read cyclically, so the initial bar indicates that a mirror
bisects the edge joining the last and first vertices.  In this example,
every edge is bisected by some mirror.

The coordinates are always chosen so that the Weyl vector $\rho$ has the form
$({\rm positive},0,0)$.  The first coordinate can be computed from the top left
corner of the inner product matrix (call it $E$).  One always has
$\rho=(-1/2E,0,0)$ and $\rho^2=1/4E$.  In this example we have
$E=-1/4$, so $\rho=(2,0,0)$ and $\rho^2=-1$.

The final piece of information is that the root lattice $L$ (the span
of the simple roots) is a copy of the lattice $L_{16.8}$ from \cite{Allcock-rk3table}.
If one wants to know more about this lattice then one must consult its
entry in \cite{Allcock-rk3table}, using the conventions explained in \cite[section~3.1]{Allcock-rk3}.
Among the information there is that the primes dividing $\det L$ are
$2$, $3$ and $5$, with
\begin{align*}
L\tensor\Z_2&{}\iso1^{-2}_{\rm II}4^1_1\\
L\tensor\Z_3&{}\iso1^{-1}3^{-2}\\
L\tensor\Z_5&{}\iso1^{-2}5^{-1}
\end{align*}
in the Conway-Sloane notation for $p$-adic lattices.  It follows that $\det L=-4\cdot3^2\cdot5=-180$.
Also recorded is the string of ``corner symbols'' for this lattice, namely $12^*_2
60^b_2 2_6 6^b_2$.  This gives information about the reflection group
of $L$;  as we saw above, this is larger than 
the one got from the roots comprising $\Pi_{4,5}$.  In particular, $L$ has
$4$ simple roots, of norms $12$, $60$, $2$ and $6$ in cyclic order
around the chamber, with the angles given by the subscripts, here
$\pi/2$, $\pi/2$, $\pi/6$ and $\pi/2$.  The superscripts are the 
arcane part of the notation and describe how to reconstruct $L$ from
the two simple roots given at a corner.  One looks for the simplest
superscript, here the missing one in $2_66$, which says that $L$ is
the direct sum of the $A_2$ root lattice and its orthogonal
complement.  Since $\det L=-180$ and $\det A_2=3$,
the orthogonal complement is spanned by a vector of norm $-60$.  So
$L\iso A_2\oplus\spanof{-60}$.

\ArxivOrBLMS{%
\bigskip
{\it Example:\/}
Consider the entry for $\Pi:=\Pi_{12,4}$, which reads
\begin{center}
\setbox\matricesbox=\hbox{%
{$\left[\!\llap{\phantom{%
\begingroup \smaller\smaller\smaller\begin{tabular}{@{}c@{}}%
\phantom{0}\\\phantom{0}\\\phantom{0}\\\phantom{0}
\end{tabular}\endgroup%
}}\right.$}%
\begingroup \smaller\smaller\smaller\begin{tabular}{@{}c@{}}%
-6\\\phantom{0}\\\phantom{0}\\\phantom{0}
\end{tabular}\endgroup%
\kern3pt%
\begingroup \smaller\smaller\smaller\begin{tabular}{@{}c@{}}%
\phantom{0}\\1\\\phantom{0}\\\phantom{0}
\end{tabular}\endgroup%
\kern3pt%
\begingroup \smaller\smaller\smaller\begin{tabular}{@{}c@{}}%
\phantom{0}\\\phantom{0}\\1\\\phantom{0}
\end{tabular}\endgroup%
\kern3pt%
\begingroup \smaller\smaller\smaller\begin{tabular}{@{}c@{}}%
\phantom{0}\\\phantom{0}\\\phantom{0}\\1
\end{tabular}\endgroup%
{$\left.\llap{\phantom{%
\begingroup \smaller\smaller\smaller\begin{tabular}{@{}c@{}}%
\phantom{0}\\\phantom{0}\\\phantom{0}\\\phantom{0}
\end{tabular}\endgroup%
}}\!\right]$}%
{$\left[\!\llap{\phantom{%
\begingroup \smaller\smaller\smaller\begin{tabular}{@{}c@{}}%
0\\0\\0\\0
\end{tabular}\endgroup%
}}\right.$}%
\begingroup \smaller\smaller\smaller\begin{tabular}{@{}c@{}}%
2\\3\\1\\-4
\end{tabular}\endgroup%
{$\left.\llap{\phantom{%
\begingroup \smaller\smaller\smaller\begin{tabular}{@{}c@{}}%
0\\0\\0\\0
\end{tabular}\endgroup%
}}\!\right]$}%
}%
\ifdim\wd\matricesbox>\halfwidth\myboxwidth=\hsize\else\myboxwidth=\halfwidth\fi
\leavevmode
\vbox{%
\ifdim\myboxwidth=\hsize
\setbox\onelinebox=\hbox{%
\vbox{\hbox{%
$\Pi_{12,4}=A_{3,III}=\hbox{GN}_{60}$ spans $L_{7.9}$%
}\hbox{%
$\slashthree\slashinfty\slashthree\slashinfty\slashthree\slashinfty\slashthree\slashinfty\slashthree\slashinfty\slashthree\slashinfty\rtimes D_{12}$%
}%
}%
\hfill\copy\matricesbox
}%
\ifdim\wd\onelinebox>\myboxwidth
\hbox to \myboxwidth{%
$\Pi_{12,4}=A_{3,III}=\hbox{GN}_{60}$ spans $L_{7.9}$%
\hfil
$\slashthree\slashinfty\slashthree\slashinfty\slashthree\slashinfty\slashthree\slashinfty\slashthree\slashinfty\slashthree\slashinfty\rtimes D_{12}$%
}%
\box\matricesbox
\else
\hbox to \myboxwidth{%
\unhbox\onelinebox
}%
\fi
\else
\hbox{%
$\Pi_{12,4}=A_{3,III}=\hbox{GN}_{60}$ spans $L_{7.9}$%
\hfil}%
\hbox{%
$\slashthree\slashinfty\slashthree\slashinfty\slashthree\slashinfty\slashthree\slashinfty\slashthree\slashinfty\slashthree\slashinfty\rtimes D_{12}$%
\hfil}%
\box\matricesbox
\fi
}%
\end{center}
First we learn that $\Pi$ has been studied before by
Gritsenko and Nikulin; they called it $A_{3,III}$ in \cite{GN-Igusa}
and \cite{GN-AFaLKMAsI}, and it is number~$60$ on their list of root
systems in \cite{GN-AFaLKMAsI}.  See
section~\ref{sec-Gritsenko-Nikulin} for more on the correspondences
between our work and theirs.

The main difference between this example and the previous one is that
here the inner product and simple roots are described in $\R^{3,1}$
rather than $\R^{2,1}$.  We do this whenever the symmetry group
contains an element of order~$3$.  The roots lies in a
$3$-dimensional subspace: their last three coordinates always sum to
zero, as for this example's displayed root $(2,3,1,-4)$.  When the
symmetry group is $D_{12}$, as here, the symmetries of the root system
are generated by the permutations of the last three coordinates, and
the simultaneous negation of all three of them.  As before, the full set of
simple roots is obtained by applying these to the displayed roots.  So
in this example the simple roots are  
$(2,3,1,-4)$,
$(2,-4,3,1)$,
$(2,1,-4,3)$,
$(2,1,3,-4)$,
$(2,3,-4,1)$,
$(2,-4,1,3)$,
$(2,-3,-1,4)$,
$(2,4,-3,-1)$,
$(2,-1,4,-3)$,
$(2,-1,-3,4)$,
$(2,-3,4,-1)$ and
$(2,4,-1,-3)$.
The reason we use four coordinates, with the last three summing to zero,
is to make these symmetries visible.

The only other new ingredient in this example is that  the
vertical bars in
$\slashthree\slashinfty\slashthree\slashinfty\slashthree\slashinfty\slashthree\slashinfty\slashthree\slashinfty\slashthree\slashinfty$
pass through the numbers rather than between them.  Such bars
indicates mirrors of $\Aut\Pi$ passing through vertices
of the chamber, rather than bisecting edges of it.
}{
}

\bigskip
The things the reader may encounter in the table that we have not
illustrated in 
\ArxivOrBLMS{these examples}{this example} are the following.  
(1) Sometimes distinct simple root systems
have isometric Weyl chambers, indicated by ``(shared)''.  
(2)
If the symmetry group of the simple root system is something other than
$D_4$\ArxivOrBLMS{ or $D_{12}$}{}, 
then the description of the symmetries will be different.  But it will
still be ``universal'' in the sense that all simple root systems with the same
indicated symmetry
group have exactly the same automorphisms.
\ArxivOrBLMS{}{%
(3) Some of the
lattices have alternate names, referring to work of Gritsenko and
Nikulin.  (4) Sometimes the vertical bars pass through the numbers
rather than between them, such as $\slashthree6\slashinfty6$ for
$\Pi_{4,10}$, indicating mirrors of $\Aut\Pi$ that pass through
vertices of the chamber.}
Now we proceed to the full details.

\medskip
{\it Other names:\/} Some of the simple root systems have already been named
by Gritsenko-Nikulin.  
This accounts for names like ``$A_{1,II}$'' and ``${\rm GN}_{25}$''.
See the next
section for details and further references.

{\it Root lattice:\/} We chose the scale at which the root
lattice $L$ is  unscaled.  Because
the Weyl group of $\Pi$ has finite area chamber, $L$ is
reflective, meaning that its reflections generate a finite
index subgroup of $\orthogonalgroup(L)$.  Therefore $L$ is one of the lattices
$L_{k.l}$ from the classification of these lattices given in
\cite{Allcock-rk3table} (and proven in \cite{Allcock-rk3}).  We
indicate which one; see the end of this section for some comments on
how we worked this out.  The reader may consult
\cite{Allcock-rk3table} for further information, such as the Conway-Sloane
symbol for the genus of~$L$.

Some of the $L_{k.l}$ occurred many times.  The ten most frequently
appearing lattices are the following, where we have written $A_2[m]$
for the triangular lattice with inner product scaled by $m$, so that
its six minimal vectors have norm $2m$.
\begin{center}
\begin{tabular}{r@{${}\times{}$}lc@{\kern60pt}r@{${}\times{}$}lc}
$\latticeHQ$&\latticeHQcount&$\spanof{-1}\oplus A_2[15]$&
$\latticeHP$&\latticeHPcount&$\spanof{-5}\oplus A_2[6]$\\
$\latticeHO$&\latticeHOcount&$\spanof{-4,36,36}$, glued&
$\latticeHN$&\latticeHNcount&$\spanof{-1,6,6}$\\
$\latticeHM$&\latticeHMcount&$\spanof{-4}\oplus A_2[3]\phantom{1}$&
$\latticeHL$&\latticeHLcount&$\spanof{-3}\oplus A_2[2]$\\
$\latticeHK$&\latticeHKcount&$\spanof{-1}\oplus A_2[10]$&
$\latticeHJ$&\latticeHJcount&$\spanof{-1}\oplus A_2[7]$\\
$\latticeHI$&\latticeHIcount&$\spanof{-1}\oplus A_2[5]\phantom{1}$&
$\latticeHH$&\latticeHHcount&$\spanof{-5}\oplus A_2[3]$
\end{tabular}
\end{center}
We found the stated descriptions of these
lattices by referring to the ``corner symbols'' in
\cite{Allcock-rk3table} (explained in \cite{Allcock-rk3}), which made
the direct sum decompositions visible.  The gluing
for $\latticeHO$ means the adjunction of the vector
$(\frac{1}{2},\frac{1}{2},\frac{1}{2})$.

It is remarkable that $\latticeHQ$ occurs so frequently---by itself it
accounts for more than half of the table entries with compact chamber.
Looking at the lists of roots in these systems also reveals a certain
sameness: many of the root systems differ only slightly.   These ten
lattices are anisotropic, hence have compact Weyl chambers, except for
$\latticeHO$.  We mentioned in the introduction that
$\orthogonalgroup(\latticeHO)\sset\orthogonalgroup(\Z^{2,1})$; this is because $3$-filling
(\cite[\S1.1]{Allcock-rk3}) the even sublattice of $\latticeHO$ and
then rescaling yields $L_{4,1}=\Z^{2,1}$.  Simple roots for this
lattice correspond to our $\Pi_{3,27}$,
introduced by Feingold-Frenkel \cite{Feingold-Frenkel}.

{\it Chamber angles:\/} We give the list of angles $\pi/(\hbox{2, 3,
  4, 6 or $\infty$})$ at the vertices of the chamber in cyclic order:
the first vertex is $\a_1^\perp\cap\a_2^\perp$ and the last is
$\a_n^\perp\cap\a_1^\perp$.  Of the $\numsystems$ Weyl chambers
listed, 
$\numrightangled$ have all right angles, $\numcpt$ are compact, and
$\numnoncpt$ are noncompact.  Also,
{\NumIdealWeyls} have all ideal vertices, meaning vertices on the
boundary of the hyperbolic plane (\IdealWeyls).
Finally, {\NumRegularWeyls} of the chambers are
regular polygons (\RegularWeyls).   The regular polygons arising this way are the
ideal triangle (3 times), the ideal square (twice), the ideal hexagon, the square with
angles $\pi/3$, and the hexagons with angles $\pi/2$ and $\pi/3$.  In
four of these cases, $\Aut\Pi$ is strictly smaller than the isometry
group of the chamber.

The (orbifold) Euler characteristic and therefore the
area of the chamber are easy to find from the angles.  The smallest
Weyl chamber is shared by the Weyl groups of \withmax, with Euler
characteristic $\maxeuler$, and the largest chamber is that of
\withmin, whose Weyl group has Euler characteristic~$\mineuler$.

{\it Chamber isometries:\/} If $\Pi$ has nontrivial isometries then the list of angles is followed by ``$\rtimes C_m$'' or
``$\rtimes D_m$''.  The subscript $m$ is $|\Aut\Pi\,|$
and we write $D$ (resp.~$C$) when there are (resp.\ are not)
reflections in this finite group.  Note the distinction between $C_2$
and $D_2$.  In the dihedral case, we have also indicated how the
mirrors of the reflections meet the boundary of the chamber.  A
$|$ through a digit (i.e., $\slashtwo$, $\slashthree$,
$\slashinfty$) means that a mirror passes through that vertex, and a
$|$ between two digits means that a mirror bisects the edge
joining the two corresponding vertices.  Because of the cyclic
ordering, a $|$ before the string of digits means that a mirror
bisects the edge $\a_1^\perp$.

{\it Shared Weyl chambers:\/} Sometimes the Weyl groups of distinct
simple root systems are conjugate in $\orthogonalgroup(2,1)$.  When this happens we include
the note ``(shared)'' in the table.  We have numbered
the roots for our simple root systems so that the edges of the two
Weyl chambers
correspond in the obvious way.  That is, if
$\alpha_1,\dots,\alpha_n\in\R^{2,1}$ and
$\alpha_1',\dots,\alpha_n'\in\R^{2,1}$ have isometric Weyl chambers,
then there is an isometry of $\R^{2,1}$ sending each $\a_i$ to a
scalar multiple of $\a_i'$.  (In the language of
\cite{Nikulin-Reflection-groups-in-hyperbolic-spaces-and-the-denominator-formula-for-LKMLAs},
the simple root
systems are twisted to each other and these scalars are the twisting
coefficients, up to a common scalar.)
For ease of cross-reference, the simple root systems that share their chamber
(and with which other systems)
are listed in table~\ref{tab-shared-weyl-chambers}.

\begin{table}
\begin{tabular}{ll}
\sharedAA&\sharedAB\\
\sharedAC&\sharedAD\\
\sharedAE&\sharedAF\\
\sharedAG&\sharedAH\\
\sharedAI&\sharedAJ\\
\sharedAK&\sharedAL\\
\sharedAM&\sharedAN\\
\sharedAO&\sharedAP\\
\sharedAQ&\sharedAR\\
\sharedAS&\sharedAT\\
\sharedAU&\sharedAV\\
\sharedAW&\sharedAX\\
\sharedAY&\sharedAZ\\
\sharedBA&\sharedBB\\
\sharedBC&\sharedBD\\
\sharedBE&\sharedBF\\
\sharedBG&\sharedBH\\
\sharedBI&\sharedBJ\\
\sharedBK&\sharedBL\\
\sharedBM&\sharedBN\\
\sharedBO&\sharedBP\\
\sharedBQ&\sharedBR\\
\sharedBS&\sharedBT\\
\sharedBU&\sharedBV
\end{tabular}
\smallskip
\caption{Equivalence classes of simple root systems, under the relation of
  having isometric chambers.  For example, the last entry means that
  $\Pi_{6,30}$ and $\Pi_{6,35}$ share a chamber.}
\label{tab-shared-weyl-chambers}
\end{table}

{\it Coordinates compatible with $\Aut\Pi$:\/} In every case we took
$-\rho/2\rho^2$ as the first basis vector, for reasons explained below.
\ArxivOrBLMS{%
When $\Pi$ doesn't admit a symmetry of order~$3$ then we extended this
to a basis of $L\tensor\Q$ by choosing two
more vectors, in such a way that $\Aut\Pi$ is made visible, as
described below.
When $\Pi$ does admit a symmetry of order~$3$, then we describe vectors in
$\rho^\perp$ by three-component vectors summing to $0$. That is, we
represent elements of $L$ by four-component vectors whose last three
coordinates sum to zero.  This is also to make $\Aut\Pi$ visible.  The
first displayed matrix is the inner product matrix of our basis
for $\R^{2,1}$ or $\R^{3,1}$, and the roots (see below) are specified
as vectors in $\R^{2,1}$ or $\R^{3,1}$.  We have already described
$\Aut\Pi$ in terms of its action on the chamber, and now we make
concrete its action on $\R^{2,1}$ or $\R^{3,1}$.

First we describe $\Aut\Pi$ when using only three coordinates.
If it 
is $D_2$ then the only symmetry is negation of the last
coordinate.  If it is $C_2$ then the only symmetry is the
negation of the last two coordinates.  If it is $D_4$ then $\Aut\Pi$
is generated by these two transformations.  If it is $C_4$ or $D_8$
then the last two basis vectors are orthogonal and have the same norm,
and the obvious rotations of order~$4$ are symmetries of~$\Pi$.  In the
$D_8$ case, $\Aut\Pi$ is generated by these and the exchange of
the last two coordinates. 
}{%
We extended this to a basis of $L\tensor\Q$ by choosing two
more vectors, and
the inner product matrix of this basis is the first printed matrix.
We chose this basis in such a way that $\Aut\Pi$'s action on
$\R^{2,1}$ is visible.  Namely,
if $\Aut\Pi$
is $D_2$ then the only symmetry is negation of the last
coordinate.  If it is $C_2$ then the only symmetry is the
negation of the last two coordinates.  If it is $D_4$ then $\Aut\Pi$
is generated by these two transformations.  If it is $C_4$ or $D_8$
then the last two basis vectors are orthogonal and have the same norm,
and the obvious rotations of order~$4$ are symmetries of~$\Pi$.  In the
$D_8$ case, $\Aut\Pi$ is generated by these and the exchange of
the last two coordinates. 

(Unlike the printed cases, some simple root systems in the unprinted portion of the table possess
symmetries of order~$3$.  Then a different sort of
coordinate system is required in order to make the symmetries visible.  
These cases appear in  \cite{Allcock-arxiv}, together with an explanation
of their coordinates.)
}

\ArxivOrBLMS{%
Now we describe $\Aut\Pi$ when using four coordinates.  In every case
it contains the cyclic permutations of the last three coordinates.  If
it is $D_6$ then one
enlarges this to all permutations of these coordinates.  
If it is
$C_6$ then instead one adjoins the simultaneous
negation of the last three coordinates.  
If it is $D_{12}$ then one makes both these
enlargements.
}{}%

{\it Inner product matrix and Weyl vector:\/} The first displayed
matrix is the inner product matrix of our basis for 
$\R^{2,1}$\ArxivOrBLMS{ or $\R^{3,1}$}{}.  It is invariant under the symmetries just described.  Recall
that our first basis vector was $-\rho/2\rho^2$.  The top left entry
(call it $E$) of the inner product matrix is the norm of this vector,
namely $1/4\rho^2$.  It follows that $\rho^2=1/4E$ and
$\rho=(-1/2E,0,\dots,0)$.

{\it The simple roots and their norms:\/} The columns of the second
displayed matrix are some of the simple roots.  Their norms are the
top row of the matrix.  This pleasant circumstance arises from the
defining property $\rho\cdot\alpha=-\alpha^2/2$ of $\rho$, and is the
reason for our curious choice of $-\rho/2\rho^2$ as the first basis
vector.

In the absence of symmetry, we have printed all of $\a_1,\dots,\a_n$.  If $\Aut\Pi=C_m$ then we have printed
$\a_1,\dots,\a_{n/m}$.  So suppose $\Aut\Pi=D_m$.  We explained above
how we used $|$'s to indicate how the mirrors of $\Aut\Pi$ meet the boundary of the chamber.  We printed
only the roots corresponding to the edges involved in the portion of
the boundary of the chamber lying strictly between the first two displayed
$|$'s.  For example, for $\slashtwo44\slashtwo44\semidirect D_2$ we
would display $\alpha_2,\alpha_3,\alpha_4$ because the first $|$
represents the vertex $\alpha_1^\perp\cap\alpha_2^\perp$ and the
second represents $\alpha_4^\perp\cap\alpha_5^\perp$.  And for
$24|424|4\semidirect D_2$ we would display $\alpha_3,\dots,\alpha_6$
since the first~$|$ corresponds to $\alpha_3$
and the second to $\alpha_6$.  (It would have
been nice to be able to print the table so that the roots displayed
were always some initial segment of $\a_1,\dots,\a_n$.  But this is
incompatible with our convention that if two simple root systems share a
chamber then their roots correspond in the obvious way.)

In all cases the full simple root system is got from the displayed roots by
applying the automorphisms described above.

{\it Generalized Cartan Matrix and Dynkin Diagram:\/} The generalized
Cartan matrix $A$ can be extracted from the inner product matrix of
the simple roots, which can be computed from the inner product matrix
of our basis, together with the simple roots.  The formula is
$A_{ij}=2\a_i\cdot\a_j/\a_i^2$.  If every pair of faces of the chamber
meet (in hyperbolic space or at its boundary) then it is conventional
to record this matrix as the Dynkin diagram.  In our situation, this
happens only when the chamber is a triangle%
\ArxivOrBLMS{}{, which means that it concerns the unprinted part of
  the table, which appears in \cite{Allcock-arxiv}}.
In this case one can read off the Dynkin diagram
without having to do the calculations we have just outlined.  For
example, $\Pi_{3,21}$'s chamber has angles $\pi/4$, $\pi/4$ and
$\pi/\infty$.  So two of the bonds in the Dynkin diagram are double
bonds and the third is a heavy bond.  The norms of the roots are the
top row of the second matrix: $1$, $2$ and~$4$.  The arrow on each
edge of the Dynkin diagram points from the larger-norm root to the
smaller-norm root.  It follows that the diagram is number~54 in
\cite[table~2]{Carbone-etal}.

\ArxivOrBLMS{}{%
Complete information about the three-root simple root systems can be
derived from their Dynkin diagrams, which are the symmetrizable cases of
\cite{Carbone-etal}.  So our classification is only new for simple
root systems with four or more simple roots.  This is the reason that
our abridged table displays those with four roots.}

\bigskip
We close this section with a remark on recognizing $\Pi$'s root
lattice.  In all except {\numgenusnotenough} cases it was enough to
work out  $L$'s genus and observe that only one
lattice in \cite{Allcock-rk3table} has that genus.  In the
remaining cases, two lattices $L_{k.l}$ from \cite{Allcock-rk3table} have that
genus and we had to determine which one was $L$.  In
{\numgenusalmostenough} of the cases ({\GenusAlmostEnough}) one of
the possibilities could be excluded because it lacked any
roots of some norm appearing among the norms of the roots in $\Pi$.
(Because $\Pi$'s roots are not assumed primitive, while roots were
assumed primitive in \cite{Allcock-rk3table}, we had to divide some
of $\Pi$'s roots by~$2$ before making this comparison.)  
We remark that $L_{10.1}$ 
occurs seven times,
accounting for more than half
of these {\numgenusalmostenough} cases.

Two
cases remained.
For
$\Pi=\hardlatticerecog$ it turned out that $\Pi$'s inner product
matrix coincides with that of the simple roots for $L_{10.9}$.
For $\Pi=\veryhardlatticerecog$ it turned out that $\Pi$ admits a
reflective symmetry, with $W(\Pi)$ having index~$2$ in a larger
Coxeter group with the same root lattice.  And a set of simple roots
for this larger group also admits a reflective symmetry, so again we
can enlarge the Coxeter group.  We write $W'$ for the resulting
index~$4$ supergroup of $W$.  
The roots for the extra reflections
can be taken to lie in $L$, and one can check that 
$W'$ has four simple roots, whose span is all of~$L$.  Finally,
it happens that the Gram matrix for $W'$'s simple roots coincides with that
of the simple roots for $L_{10.4}$.   So $L\iso L_{10.4}$.
Curiously, the root
lattices $L_{10.9}$ and $L_{10.4}$ in these two difficult cases have the
same genus.

\section{Relation to work of Gritsenko-Nikulin}
\label{sec-Gritsenko-Nikulin}

Gritsenko and Nikulin have published three lists of simple root systems that
we make explicit reference to in the table.  
\ArxivOrBLMS{}{The information we give here about the correspondence
  between their root systems and ours contains  many references
  to the unprinted part of the table; see \cite{Allcock-arxiv} for the
  full table.\par}
Their first enumeration
appears in theorem~4.1 of \cite{GN-Igusa} and consists of all twelve
rank~3 hyperbolic simple root systems that have equal root norms and
admit a timelike Weyl vector, and whose chamber is noncompact of finite area.
(They also consider the case of a lightlike Weyl vector, which yields
a thirteenth root system.)
See theorem~1.3.1 of \cite{GN-AFaLKMAsI} for the proof of this
classification.  Their and
our names for these simple root systems correspond as follows:
\begin{center}
\begin{tabular}{r@{${}\leftrightarrow{}$}lr@{${}\leftrightarrow{}$}lr@{${}\leftrightarrow{}$}lr@{${}\leftrightarrow{}$}l}
$A_{1,0}$&$\myAonezero$%
&$A_{1,I}$&$\myAoneI$%
&$A_{1,II}$&$\myAoneII$%
&$A_{1,III}$&$\myAoneIII$%
\\
$A_{2,0}$&$\myAtwozero$%
&$A_{2,I}$&$\myAtwoI$%
&$A_{2,II}$&$\myAtwoII$%
&$A_{2,III}$&$\myAtwoIII$%
\\
$A_{3,0}$&$\myAthreezero$%
&$A_{3,I}$&$\myAthreeI$%
&$A_{3,II}$&$\myAthreeII$%
&$A_{3,III}$&$\myAthreeIII$%
\end{tabular}
\end{center}
The first detailed study of any hyperbolic KMA was by Feingold-Frenkel
\cite{Feingold-Frenkel}, whose $A$ is $A_{1,0}$.  A detailed treatment of the
automorphic corrections to this KMA and the one for $A_{1,I}$ appears
in \cite{GN-Igusa}.  These two
cases are also called $G_1$ and $G_2$ in
\cite{GN-Siegel-Automorphic-form-corrections-of-some-LKMLAs}. Examples
1 and~2 in \cite{GN-K3-surfaces-Lorentzian-KMAs-and-Mirror-Symmetry}
are the automorphic corrections of the KMAs associated to $A_{1,II}$
and $A_{2,II}$.

Their second enumeration appears in table~1 of \cite{GN-AFaLKMAsI} and
consists of the sixty rank~$3$ hyperbolic simple root systems of elliptic
type that admit a Weyl vector and are twisted to a hyperbolic simple root
system all of whose roots have norm~$2$.  (This means that one can
rescale all the roots to have norm~$2$ and still have a simple root system.)
The completeness of their listing is conjecture~1.2.2 of
\cite{GN-AFaLKMAsI}.  Our enumeration proves this conjecture.  For
each $k=1,\dots,60$, we call the $k$th member of their list
``$\hbox{GN}_k$''.  In order of increasing $k$, these simple root systems are
\begin{center}
\begin{tabular}{llllllllll}
$\GNa $&$\GNb $&$\GNc $&$\GNd $&$\GNe $&$\GNf $&$\GNg $&$\GNh $&$\GNi $&$\GNj $\\
$\GNk $&$\GNl $&$\GNm $&$\GNn $&$\GNo $&$\GNp $&$\GNq $&$\GNr $&$\GNs $&$\GNt $\\
$\GNu $&$\GNv $&$\GNw $&$\GNx $&$\GNy $&$\GNz $&$\GNaa$&$\GNab$&$\GNac$&$\GNad$\\
$\GNae$&$\GNaf$&$\GNag$&$\GNah$&$\GNai$&$\GNaj$&$\GNak$&$\GNal$&$\GNam$&$\GNan$\\
$\GNao$&$\GNap$&$\GNaq$&$\GNar$&$\GNas$&$\GNat$&$\GNau$&$\GNav$&$\GNaw$&$\GNax$\\
$\GNay$&$\GNaz$&$\GNba$&$\GNbb$&$\GNbc$&$\GNbd$&$\GNbe$&$\GNbf$&$\GNbg$&$\GNbh$
\end{tabular}
\end{center}
Of the sixty, sixteen have all root norms equal, and twelve
of these are the cases $A_{1,0},\dots,A_{3,III}$ referred to above.
The remaining four have compact Weyl chambers, with Gram matrices
$B_1,\dots,B_4$ listed in theorem~1.5.6 of \cite{GN-On-classification-of-LKMAs}.  The
correspondence with our simple root systems is
\begin{center}
\begin{tabular}{r@{${}\leftrightarrow{}$}lr@{${}\leftrightarrow{}$}lr@{${}\leftrightarrow{}$}lr@{${}\leftrightarrow{}$}l}
$B_1$&$\myBone$
&$B_2$&$\myBtwo$
&$B_3$&$\myBthree$
&$B_4$&$\myBfour$
\end{tabular}
\end{center}

The third enumeration \cite{GN-On-classification-of-LKMAs} is really an
enumeration of automorphic corrections to Kac-Moody algebras rather
than simple root systems.
That is, they classified the pairs  $(\Pi,\xi)$ where $\Pi$ is a
simple root system satisfying (i)--(iii) of
theorem~\ref{thm-main} and consists of roots of some 
lattice of the form
$L_t=\spanof{2t}\oplus\hypcell{1}$, and 
$\xi$ is a suitable automorphic
form for a certain subgroup of $\orthogonalgroup\bigl(L_t\oplus\hypcell{1}\bigr)$.
They showed that there are exactly 29
such pairs $(\Pi,\xi)$.  
Of their~$29$ simple root systems, $23$ have timelike Weyl vectors
and therefore appear in our tables.  
The
simple root systems $A_{1,II}$, $A_{2,II}$ and $A_{3,II}$ mentioned above
appear twice each and $A_{1,0}$, $A_{2,0}$ and $A_{3,0}$ appear once
each.  The remaining $14$ are the following.
$$
\begin{tabular}{r@{${}\leftrightarrow{}$}lr@{}l@{\kern19pt}r@{${}\leftrightarrow{}$}lr@{}l}
$A_{1,I,\zerobar}$&$\myAoneIzerobar$&$L_{2.3}$&${}\iso\spanof{1}\oplus\hypcell{2}$
&$A_{2,I,\zerobar}$&$\myAtwoIzerobar$&$L_{1.9}$&${}\iso\spanof{1}\oplus\hypcell{4}$\\
$A_{2,II,\onebar}$&$\myAtwoIIonebar$&$L_{1.9}$&
&$A_{2,I,\onebar}$&$\myAtwoIonebar$&$L_{1.9}$\\
$A_{2,0,\onebar}$&$\myAtwozeroonebar$&$L_{1.4}$&${}\iso\spanof{1}\oplus\hypcell{1}$
&$A_{3,I,\zerobar}$&$\myAthreeIzerobar$&$L_{7.7}$&${}\iso\spanof{-3}\oplus A_2[2]$\\
$A_{3,II,\onebar}$&$\myAthreeIIonebar$&$L_{7.7}$&
&$A_{3,I,\onebar}$&$\myAthreeIonebar$&$L_{7.7}$\\
$A_{3,0,\onebar}$&$\myAthreezeroonebar$&$L_{7.6}$&${}\iso\spanof{-6}\oplus A_2$
&$A_{4,I,\zerobar}$&$\myAfourIzerobar$&$L_{140.4}$&${}\iso\spanof{1}\oplus\hypcell{8}$\\
$A_{4,II,\onebar}$&$\myAfourIIonebar$&$L_{140.4}$&
&$A_{4,I,\onebar}$&$\myAfourIonebar$&$L_{140.4}$\\
$A_{4,0,\onebar}$&$\myAfourzeroonebar$&$L_{2.3}$&
&$\xi_{0,9}^{(2)}$&$\myAxi$&$L_{148.9}$&${}\iso\spanof{1}\oplus\hypcell{18}$%
\end{tabular}
$$




\noindent
(Gritsenko and Nikulin had found all of these automorphic corrections
previously, in 
\cite{GN-Igusa},
\cite{GN-Siegel-Automorphic-form-corrections-of-some-LKMLAs}
and
\cite{GN-AFaLKMAsII}.  Also, they used
signature $(1,2)$ where we use $(2,1)$, and in the last case
they did not name the generalized Cartan matrix.  We used their name
for the automorphic form.)

The use of ``root lattice'' in \cite{GN-On-classification-of-LKMAs}
differs slightly from ours.  There, it was the lattice $L_t$ fixed
beforehand, from which one chooses a simple root system that doesn't
necessarily span~$L_t$.  So we have also described the
root lattice in our sense, meaning the span of the roots, rescaled to
be unscaled.  For completeness we also do this for
$A_{1,0},\dots,A_{3,0}$ and $A_{1,II},\dots,A_{3,II}$:
$$
\begin{tabular}{r@{${}\leftrightarrow{}$}lr@{}l@{\kern25pt}r@{${}\leftrightarrow{}$}lr@{}l}
$A_{1,0}$&$\myAonezero$&$L_{2.2}$&${}\iso\spanof{2}\oplus\hypcell{1}$
&$A_{1,II}$&$\myAoneII$&$L_{2.3}$&${}\iso\spanof{1}\oplus\hypcell{2}$\\
$A_{2,0}$&$\myAtwozero$&$L_{1.4}$&${}\iso\spanof{1}\oplus\hypcell{1}$
&$A_{2,II}$&$\myAtwoII$&$L_{1.9}$&${}\iso\spanof{1}\oplus\hypcell{4}$\\
$A_{3,0}$&$\myAthreezero$&$L_{7.6}$&${}\iso\spanof{-6}\oplus A_2$
&$A_{3,II}$&$\myAthreeII$&$L_{7.7}$&${}\iso\spanof{-3}\oplus A_2[2]$
\end{tabular}
$$

\bigskip
\centerline{{\bf Table of Simple Root Systems} (see theorem~\ref{thm-main})}%
\ArxivOrBLMS{}{\nopagebreak\centerline{(for root systems with $3,5,6,7,\dots,24$
    simple roots, see \cite{Allcock-arxiv})}}
\leftskip=0pt
\parindent=0pt
\nopagebreak\medskip
\ArxivOrBLMS{%
\leavevmode
\setbox\matricesbox=\hbox{%
{$\left[\!\llap{\phantom{%
\begingroup \smaller\smaller\smaller\begin{tabular}{@{}c@{}}%
\phantom{0}\\\phantom{0}\\\phantom{0}\\\phantom{0}
\end{tabular}\endgroup%
}}\right.$}%
\begingroup \smaller\smaller\smaller\begin{tabular}{@{}c@{}}%
-1/3\\\phantom{0}\\\phantom{0}\\\phantom{0}
\end{tabular}\endgroup%
\kern3pt%
\begingroup \smaller\smaller\smaller\begin{tabular}{@{}c@{}}%
\phantom{0}\\2/9\\\phantom{0}\\\phantom{0}
\end{tabular}\endgroup%
\kern3pt%
\begingroup \smaller\smaller\smaller\begin{tabular}{@{}c@{}}%
\phantom{0}\\\phantom{0}\\2/9\\\phantom{0}
\end{tabular}\endgroup%
\kern3pt%
\begingroup \smaller\smaller\smaller\begin{tabular}{@{}c@{}}%
\phantom{0}\\\phantom{0}\\\phantom{0}\\2/9
\end{tabular}\endgroup%
{$\left.\llap{\phantom{%
\begingroup \smaller\smaller\smaller\begin{tabular}{@{}c@{}}%
\phantom{0}\\\phantom{0}\\\phantom{0}\\\phantom{0}
\end{tabular}\endgroup%
}}\!\right]$}%
{$\left[\!\llap{\phantom{%
\begingroup \smaller\smaller\smaller\begin{tabular}{@{}c@{}}%
0\\0\\0\\0
\end{tabular}\endgroup%
}}\right.$}%
\begingroup \smaller\smaller\smaller\begin{tabular}{@{}c@{}}%
1\\2\\-1\\-1
\end{tabular}\endgroup%
{$\left.\llap{\phantom{%
\begingroup \smaller\smaller\smaller\begin{tabular}{@{}c@{}}%
0\\0\\0\\0
\end{tabular}\endgroup%
}}\!\right]$}%
}%
\ifdim\wd\matricesbox>\halfwidth\myboxwidth=\hsize\else\myboxwidth=\halfwidth\fi
\vbox{%
\ifdim\myboxwidth=\hsize
\setbox\onelinebox=\hbox{%
\vbox{\hbox{%
$\Pi_{3,1}=A_{1,II}=\hbox{GN}_{25}$ spans $L_{2.3}$%
}\hbox{%
$|\slashinfty|\slashinfty|\slashinfty\rtimes D_{6}$ (shared)%
}%
}%
\hfill\copy\matricesbox
}%
\ifdim\wd\onelinebox>\myboxwidth
\hbox to \myboxwidth{%
$\Pi_{3,1}=A_{1,II}=\hbox{GN}_{25}$ spans $L_{2.3}$%
\hfil
$|\slashinfty|\slashinfty|\slashinfty\rtimes D_{6}$ (shared)%
}%
\box\matricesbox
\else
\hbox to \myboxwidth{%
\unhbox\onelinebox
}%
\fi
\else
\hbox to \myboxwidth{%
$\Pi_{3,1}=A_{1,II}=\hbox{GN}_{25}$ spans $L_{2.3}$%
\hfil}%
\hbox to \myboxwidth{%
$|\slashinfty|\slashinfty|\slashinfty\rtimes D_{6}$ (shared)%
\hfil}%
\box\matricesbox
\fi
}%
\hfill\discretionary{}{}{}%
\setbox\matricesbox=\hbox{%
{$\left[\!\llap{\phantom{%
\begingroup \smaller\smaller\smaller\begin{tabular}{@{}c@{}}%
\phantom{0}\\\phantom{0}\\\phantom{0}
\end{tabular}\endgroup%
}}\right.$}%
\begingroup \smaller\smaller\smaller\begin{tabular}{@{}c@{}}%
-1/16\\\phantom{0}\\\phantom{0}
\end{tabular}\endgroup%
\kern3pt%
\begingroup \smaller\smaller\smaller\begin{tabular}{@{}c@{}}%
\phantom{0}\\1/4\\\phantom{0}
\end{tabular}\endgroup%
\kern3pt%
\begingroup \smaller\smaller\smaller\begin{tabular}{@{}c@{}}%
\phantom{0}\\\phantom{0}\\2
\end{tabular}\endgroup%
{$\left.\llap{\phantom{%
\begingroup \smaller\smaller\smaller\begin{tabular}{@{}c@{}}%
\phantom{0}\\\phantom{0}\\\phantom{0}
\end{tabular}\endgroup%
}}\!\right]$}%
{$\left[\!\llap{\phantom{%
\begingroup \smaller\smaller\smaller\begin{tabular}{@{}c@{}}%
0\\0\\0
\end{tabular}\endgroup%
}}\right.$}%
\begingroup \smaller\smaller\smaller\begin{tabular}{@{}c@{}}%
2\\3\\0
\end{tabular}\endgroup%
\kern3pt%
\begingroup \smaller\smaller\smaller\begin{tabular}{@{}c@{}}%
2\\-1\\-1
\end{tabular}\endgroup%
{$\left.\llap{\phantom{%
\begingroup \smaller\smaller\smaller\begin{tabular}{@{}c@{}}%
0\\0\\0
\end{tabular}\endgroup%
}}\!\right]$}%
}%
\ifdim\wd\matricesbox>\halfwidth\myboxwidth=\hsize\else\myboxwidth=\halfwidth\fi
\vbox{%
\ifdim\myboxwidth=\hsize
\setbox\onelinebox=\hbox{%
\vbox{\hbox{%
$\Pi_{3,2}=A_{1,I}=\hbox{GN}_{12}$ spans $L_{2.1}$%
}\hbox{%
$3|3\slashinfty\rtimes D_{2}$%
}%
}%
\hfill\copy\matricesbox
}%
\ifdim\wd\onelinebox>\myboxwidth
\hbox to \myboxwidth{%
$\Pi_{3,2}=A_{1,I}=\hbox{GN}_{12}$ spans $L_{2.1}$%
\hfil
$3|3\slashinfty\rtimes D_{2}$%
}%
\box\matricesbox
\else
\hbox to \myboxwidth{%
\unhbox\onelinebox
}%
\fi
\else
\hbox to \myboxwidth{%
$\Pi_{3,2}=A_{1,I}=\hbox{GN}_{12}$ spans $L_{2.1}$%
\hfil}%
\hbox to \myboxwidth{%
$3|3\slashinfty\rtimes D_{2}$%
\hfil}%
\box\matricesbox
\fi
}%
\hfill\discretionary{}{}{}%
\setbox\matricesbox=\hbox{%
{$\left[\!\llap{\phantom{%
\begingroup \smaller\smaller\smaller\begin{tabular}{@{}c@{}}%
\phantom{0}\\\phantom{0}\\\phantom{0}
\end{tabular}\endgroup%
}}\right.$}%
\begingroup \smaller\smaller\smaller\begin{tabular}{@{}c@{}}%
-1/17\\\phantom{0}\\\phantom{0}
\end{tabular}\endgroup%
\kern3pt%
\begingroup \smaller\smaller\smaller\begin{tabular}{@{}c@{}}%
\phantom{0}\\1/34\\\phantom{0}
\end{tabular}\endgroup%
\kern3pt%
\begingroup \smaller\smaller\smaller\begin{tabular}{@{}c@{}}%
\phantom{0}\\\phantom{0}\\3/2
\end{tabular}\endgroup%
{$\left.\llap{\phantom{%
\begingroup \smaller\smaller\smaller\begin{tabular}{@{}c@{}}%
\phantom{0}\\\phantom{0}\\\phantom{0}
\end{tabular}\endgroup%
}}\!\right]$}%
{$\left[\!\llap{\phantom{%
\begingroup \smaller\smaller\smaller\begin{tabular}{@{}c@{}}%
0\\0\\0
\end{tabular}\endgroup%
}}\right.$}%
\begingroup \smaller\smaller\smaller\begin{tabular}{@{}c@{}}%
2\\5\\1
\end{tabular}\endgroup%
\kern3pt%
\begingroup \smaller\smaller\smaller\begin{tabular}{@{}c@{}}%
1\\-6\\0
\end{tabular}\endgroup%
{$\left.\llap{\phantom{%
\begingroup \smaller\smaller\smaller\begin{tabular}{@{}c@{}}%
0\\0\\0
\end{tabular}\endgroup%
}}\!\right]$}%
}%
\ifdim\wd\matricesbox>\halfwidth\myboxwidth=\hsize\else\myboxwidth=\halfwidth\fi
\vbox{%
\ifdim\myboxwidth=\hsize
\setbox\onelinebox=\hbox{%
\vbox{\hbox{%
$\Pi_{3,3}$ spans $L_{3.2}$%
}\hbox{%
$\slashthree4|4\rtimes D_{2}$ (shared)%
}%
}%
\hfill\copy\matricesbox
}%
\ifdim\wd\onelinebox>\myboxwidth
\hbox to \myboxwidth{%
$\Pi_{3,3}$ spans $L_{3.2}$%
\hfil
$\slashthree4|4\rtimes D_{2}$ (shared)%
}%
\box\matricesbox
\else
\hbox to \myboxwidth{%
\unhbox\onelinebox
}%
\fi
\else
\hbox to \myboxwidth{%
$\Pi_{3,3}$ spans $L_{3.2}$%
\hfil}%
\hbox to \myboxwidth{%
$\slashthree4|4\rtimes D_{2}$ (shared)%
\hfil}%
\box\matricesbox
\fi
}%
\hfill\discretionary{}{}{}%
\setbox\matricesbox=\hbox{%
{$\left[\!\llap{\phantom{%
\begingroup \smaller\smaller\smaller\begin{tabular}{@{}c@{}}%
\phantom{0}\\\phantom{0}\\\phantom{0}
\end{tabular}\endgroup%
}}\right.$}%
\begingroup \smaller\smaller\smaller\begin{tabular}{@{}c@{}}%
-1/25\\\phantom{0}\\\phantom{0}
\end{tabular}\endgroup%
\kern3pt%
\begingroup \smaller\smaller\smaller\begin{tabular}{@{}c@{}}%
\phantom{0}\\3/50\\\phantom{0}
\end{tabular}\endgroup%
\kern3pt%
\begingroup \smaller\smaller\smaller\begin{tabular}{@{}c@{}}%
\phantom{0}\\\phantom{0}\\9/2
\end{tabular}\endgroup%
{$\left.\llap{\phantom{%
\begingroup \smaller\smaller\smaller\begin{tabular}{@{}c@{}}%
\phantom{0}\\\phantom{0}\\\phantom{0}
\end{tabular}\endgroup%
}}\!\right]$}%
{$\left[\!\llap{\phantom{%
\begingroup \smaller\smaller\smaller\begin{tabular}{@{}c@{}}%
0\\0\\0
\end{tabular}\endgroup%
}}\right.$}%
\begingroup \smaller\smaller\smaller\begin{tabular}{@{}c@{}}%
6\\7\\1
\end{tabular}\endgroup%
\kern3pt%
\begingroup \smaller\smaller\smaller\begin{tabular}{@{}c@{}}%
2\\-6\\0
\end{tabular}\endgroup%
{$\left.\llap{\phantom{%
\begingroup \smaller\smaller\smaller\begin{tabular}{@{}c@{}}%
0\\0\\0
\end{tabular}\endgroup%
}}\!\right]$}%
}%
\ifdim\wd\matricesbox>\halfwidth\myboxwidth=\hsize\else\myboxwidth=\halfwidth\fi
\vbox{%
\ifdim\myboxwidth=\hsize
\setbox\onelinebox=\hbox{%
\vbox{\hbox{%
$\Pi_{3,4}$ spans $L_{155.1}$%
}\hbox{%
$\slashthree6|6\rtimes D_{2}$ (shared)%
}%
}%
\hfill\copy\matricesbox
}%
\ifdim\wd\onelinebox>\myboxwidth
\hbox to \myboxwidth{%
$\Pi_{3,4}$ spans $L_{155.1}$%
\hfil
$\slashthree6|6\rtimes D_{2}$ (shared)%
}%
\box\matricesbox
\else
\hbox to \myboxwidth{%
\unhbox\onelinebox
}%
\fi
\else
\hbox to \myboxwidth{%
$\Pi_{3,4}$ spans $L_{155.1}$%
\hfil}%
\hbox to \myboxwidth{%
$\slashthree6|6\rtimes D_{2}$ (shared)%
\hfil}%
\box\matricesbox
\fi
}%
\hfill\discretionary{}{}{}%
\setbox\matricesbox=\hbox{%
{$\left[\!\llap{\phantom{%
\begingroup \smaller\smaller\smaller\begin{tabular}{@{}c@{}}%
\phantom{0}\\\phantom{0}\\\phantom{0}
\end{tabular}\endgroup%
}}\right.$}%
\begingroup \smaller\smaller\smaller\begin{tabular}{@{}c@{}}%
-1/11\\\phantom{0}\\\phantom{0}
\end{tabular}\endgroup%
\kern3pt%
\begingroup \smaller\smaller\smaller\begin{tabular}{@{}c@{}}%
\phantom{0}\\3/11\\\phantom{0}
\end{tabular}\endgroup%
\kern3pt%
\begingroup \smaller\smaller\smaller\begin{tabular}{@{}c@{}}%
\phantom{0}\\\phantom{0}\\3
\end{tabular}\endgroup%
{$\left.\llap{\phantom{%
\begingroup \smaller\smaller\smaller\begin{tabular}{@{}c@{}}%
\phantom{0}\\\phantom{0}\\\phantom{0}
\end{tabular}\endgroup%
}}\!\right]$}%
{$\left[\!\llap{\phantom{%
\begingroup \smaller\smaller\smaller\begin{tabular}{@{}c@{}}%
0\\0\\0
\end{tabular}\endgroup%
}}\right.$}%
\begingroup \smaller\smaller\smaller\begin{tabular}{@{}c@{}}%
4\\3\\1
\end{tabular}\endgroup%
\kern3pt%
\begingroup \smaller\smaller\smaller\begin{tabular}{@{}c@{}}%
1\\-2\\0
\end{tabular}\endgroup%
{$\left.\llap{\phantom{%
\begingroup \smaller\smaller\smaller\begin{tabular}{@{}c@{}}%
0\\0\\0
\end{tabular}\endgroup%
}}\!\right]$}%
}%
\ifdim\wd\matricesbox>\halfwidth\myboxwidth=\hsize\else\myboxwidth=\halfwidth\fi
\vbox{%
\ifdim\myboxwidth=\hsize
\setbox\onelinebox=\hbox{%
\vbox{\hbox{%
$\Pi_{3,5}=A_{3,I,\zerobar}=\hbox{GN}_{11}$ spans $L_{7.7}$%
}\hbox{%
$\slashthree\infty|\infty\rtimes D_{2}$ (shared)%
}%
}%
\hfill\copy\matricesbox
}%
\ifdim\wd\onelinebox>\myboxwidth
\hbox to \myboxwidth{%
$\Pi_{3,5}=A_{3,I,\zerobar}=\hbox{GN}_{11}$ spans $L_{7.7}$%
\hfil
$\slashthree\infty|\infty\rtimes D_{2}$ (shared)%
}%
\box\matricesbox
\else
\hbox to \myboxwidth{%
\unhbox\onelinebox
}%
\fi
\else
\hbox to \myboxwidth{%
$\Pi_{3,5}=A_{3,I,\zerobar}=\hbox{GN}_{11}$ spans $L_{7.7}$%
\hfil}%
\hbox to \myboxwidth{%
$\slashthree\infty|\infty\rtimes D_{2}$ (shared)%
\hfil}%
\box\matricesbox
\fi
}%
\hfill\discretionary{}{}{}%
\setbox\matricesbox=\hbox{%
{$\left[\!\llap{\phantom{%
\begingroup \smaller\smaller\smaller\begin{tabular}{@{}c@{}}%
\phantom{0}\\\phantom{0}\\\phantom{0}
\end{tabular}\endgroup%
}}\right.$}%
\begingroup \smaller\smaller\smaller\begin{tabular}{@{}c@{}}%
-1/28\\\phantom{0}\\\phantom{0}
\end{tabular}\endgroup%
\kern3pt%
\begingroup \smaller\smaller\smaller\begin{tabular}{@{}c@{}}%
\phantom{0}\\1/14\\\phantom{0}
\end{tabular}\endgroup%
\kern3pt%
\begingroup \smaller\smaller\smaller\begin{tabular}{@{}c@{}}%
\phantom{0}\\\phantom{0}\\3/2
\end{tabular}\endgroup%
{$\left.\llap{\phantom{%
\begingroup \smaller\smaller\smaller\begin{tabular}{@{}c@{}}%
\phantom{0}\\\phantom{0}\\\phantom{0}
\end{tabular}\endgroup%
}}\!\right]$}%
{$\left[\!\llap{\phantom{%
\begingroup \smaller\smaller\smaller\begin{tabular}{@{}c@{}}%
0\\0\\0
\end{tabular}\endgroup%
}}\right.$}%
\begingroup \smaller\smaller\smaller\begin{tabular}{@{}c@{}}%
2\\-3\\-1
\end{tabular}\endgroup%
\kern3pt%
\begingroup \smaller\smaller\smaller\begin{tabular}{@{}c@{}}%
4\\8\\0
\end{tabular}\endgroup%
{$\left.\llap{\phantom{%
\begingroup \smaller\smaller\smaller\begin{tabular}{@{}c@{}}%
0\\0\\0
\end{tabular}\endgroup%
}}\!\right]$}%
}%
\ifdim\wd\matricesbox>\halfwidth\myboxwidth=\hsize\else\myboxwidth=\halfwidth\fi
\vbox{%
\ifdim\myboxwidth=\hsize
\setbox\onelinebox=\hbox{%
\vbox{\hbox{%
$\Pi_{3,6}$ spans $L_{3.1}$%
}\hbox{%
$\slashthree4|4\rtimes D_{2}$ (shared)%
}%
}%
\hfill\copy\matricesbox
}%
\ifdim\wd\onelinebox>\myboxwidth
\hbox to \myboxwidth{%
$\Pi_{3,6}$ spans $L_{3.1}$%
\hfil
$\slashthree4|4\rtimes D_{2}$ (shared)%
}%
\box\matricesbox
\else
\hbox to \myboxwidth{%
\unhbox\onelinebox
}%
\fi
\else
\hbox to \myboxwidth{%
$\Pi_{3,6}$ spans $L_{3.1}$%
\hfil}%
\hbox to \myboxwidth{%
$\slashthree4|4\rtimes D_{2}$ (shared)%
\hfil}%
\box\matricesbox
\fi
}%
\hfill\discretionary{}{}{}%
\setbox\matricesbox=\hbox{%
{$\left[\!\llap{\phantom{%
\begingroup \smaller\smaller\smaller\begin{tabular}{@{}c@{}}%
\phantom{0}\\\phantom{0}\\\phantom{0}
\end{tabular}\endgroup%
}}\right.$}%
\begingroup \smaller\smaller\smaller\begin{tabular}{@{}c@{}}%
-3/26\\\phantom{0}\\\phantom{0}
\end{tabular}\endgroup%
\kern3pt%
\begingroup \smaller\smaller\smaller\begin{tabular}{@{}c@{}}%
\phantom{0}\\1/26\\\phantom{0}
\end{tabular}\endgroup%
\kern3pt%
\begingroup \smaller\smaller\smaller\begin{tabular}{@{}c@{}}%
\phantom{0}\\\phantom{0}\\3/2
\end{tabular}\endgroup%
{$\left.\llap{\phantom{%
\begingroup \smaller\smaller\smaller\begin{tabular}{@{}c@{}}%
\phantom{0}\\\phantom{0}\\\phantom{0}
\end{tabular}\endgroup%
}}\!\right]$}%
{$\left[\!\llap{\phantom{%
\begingroup \smaller\smaller\smaller\begin{tabular}{@{}c@{}}%
0\\0\\0
\end{tabular}\endgroup%
}}\right.$}%
\begingroup \smaller\smaller\smaller\begin{tabular}{@{}c@{}}%
2\\-5\\-1
\end{tabular}\endgroup%
\kern3pt%
\begingroup \smaller\smaller\smaller\begin{tabular}{@{}c@{}}%
2\\8\\0
\end{tabular}\endgroup%
{$\left.\llap{\phantom{%
\begingroup \smaller\smaller\smaller\begin{tabular}{@{}c@{}}%
0\\0\\0
\end{tabular}\endgroup%
}}\!\right]$}%
}%
\ifdim\wd\matricesbox>\halfwidth\myboxwidth=\hsize\else\myboxwidth=\halfwidth\fi
\vbox{%
\ifdim\myboxwidth=\hsize
\setbox\onelinebox=\hbox{%
\vbox{\hbox{%
$\Pi_{3,7}=A_{3,0}=\hbox{GN}_{19}$ spans $L_{7.6}$%
}\hbox{%
$\slashthree\infty|\infty\rtimes D_{2}$ (shared)%
}%
}%
\hfill\copy\matricesbox
}%
\ifdim\wd\onelinebox>\myboxwidth
\hbox to \myboxwidth{%
$\Pi_{3,7}=A_{3,0}=\hbox{GN}_{19}$ spans $L_{7.6}$%
\hfil
$\slashthree\infty|\infty\rtimes D_{2}$ (shared)%
}%
\box\matricesbox
\else
\hbox to \myboxwidth{%
\unhbox\onelinebox
}%
\fi
\else
\hbox to \myboxwidth{%
$\Pi_{3,7}=A_{3,0}=\hbox{GN}_{19}$ spans $L_{7.6}$%
\hfil}%
\hbox to \myboxwidth{%
$\slashthree\infty|\infty\rtimes D_{2}$ (shared)%
\hfil}%
\box\matricesbox
\fi
}%
\hfill\discretionary{}{}{}%
\setbox\matricesbox=\hbox{%
{$\left[\!\llap{\phantom{%
\begingroup \smaller\smaller\smaller\begin{tabular}{@{}c@{}}%
\phantom{0}\\\phantom{0}\\\phantom{0}
\end{tabular}\endgroup%
}}\right.$}%
\begingroup \smaller\smaller\smaller\begin{tabular}{@{}c@{}}%
-1/19\\\phantom{0}\\\phantom{0}
\end{tabular}\endgroup%
\kern3pt%
\begingroup \smaller\smaller\smaller\begin{tabular}{@{}c@{}}%
\phantom{0}\\3/38\\\phantom{0}
\end{tabular}\endgroup%
\kern3pt%
\begingroup \smaller\smaller\smaller\begin{tabular}{@{}c@{}}%
\phantom{0}\\\phantom{0}\\3/2
\end{tabular}\endgroup%
{$\left.\llap{\phantom{%
\begingroup \smaller\smaller\smaller\begin{tabular}{@{}c@{}}%
\phantom{0}\\\phantom{0}\\\phantom{0}
\end{tabular}\endgroup%
}}\!\right]$}%
{$\left[\!\llap{\phantom{%
\begingroup \smaller\smaller\smaller\begin{tabular}{@{}c@{}}%
0\\0\\0
\end{tabular}\endgroup%
}}\right.$}%
\begingroup \smaller\smaller\smaller\begin{tabular}{@{}c@{}}%
2\\3\\1
\end{tabular}\endgroup%
\kern3pt%
\begingroup \smaller\smaller\smaller\begin{tabular}{@{}c@{}}%
6\\-10\\0
\end{tabular}\endgroup%
{$\left.\llap{\phantom{%
\begingroup \smaller\smaller\smaller\begin{tabular}{@{}c@{}}%
0\\0\\0
\end{tabular}\endgroup%
}}\!\right]$}%
}%
\ifdim\wd\matricesbox>\halfwidth\myboxwidth=\hsize\else\myboxwidth=\halfwidth\fi
\vbox{%
\ifdim\myboxwidth=\hsize
\setbox\onelinebox=\hbox{%
\vbox{\hbox{%
$\Pi_{3,8}$ spans $L_{3.4}$%
}\hbox{%
$\slashthree6|6\rtimes D_{2}$ (shared)%
}%
}%
\hfill\copy\matricesbox
}%
\ifdim\wd\onelinebox>\myboxwidth
\hbox to \myboxwidth{%
$\Pi_{3,8}$ spans $L_{3.4}$%
\hfil
$\slashthree6|6\rtimes D_{2}$ (shared)%
}%
\box\matricesbox
\else
\hbox to \myboxwidth{%
\unhbox\onelinebox
}%
\fi
\else
\hbox to \myboxwidth{%
$\Pi_{3,8}$ spans $L_{3.4}$%
\hfil}%
\hbox to \myboxwidth{%
$\slashthree6|6\rtimes D_{2}$ (shared)%
\hfil}%
\box\matricesbox
\fi
}%
\hfill\discretionary{}{}{}%
\setbox\matricesbox=\hbox{%
{$\left[\!\llap{\phantom{%
\begingroup \smaller\smaller\smaller\begin{tabular}{@{}c@{}}%
\phantom{0}\\\phantom{0}\\\phantom{0}
\end{tabular}\endgroup%
}}\right.$}%
\begingroup \smaller\smaller\smaller\begin{tabular}{@{}c@{}}%
-1/16\\\phantom{0}\\\phantom{0}
\end{tabular}\endgroup%
\kern3pt%
\begingroup \smaller\smaller\smaller\begin{tabular}{@{}c@{}}%
\phantom{0}\\3/4\\\phantom{0}
\end{tabular}\endgroup%
\kern3pt%
\begingroup \smaller\smaller\smaller\begin{tabular}{@{}c@{}}%
\phantom{0}\\\phantom{0}\\3/2
\end{tabular}\endgroup%
{$\left.\llap{\phantom{%
\begingroup \smaller\smaller\smaller\begin{tabular}{@{}c@{}}%
\phantom{0}\\\phantom{0}\\\phantom{0}
\end{tabular}\endgroup%
}}\!\right]$}%
{$\left[\!\llap{\phantom{%
\begingroup \smaller\smaller\smaller\begin{tabular}{@{}c@{}}%
0\\0\\0
\end{tabular}\endgroup%
}}\right.$}%
\begingroup \smaller\smaller\smaller\begin{tabular}{@{}c@{}}%
2\\-1\\1
\end{tabular}\endgroup%
\kern3pt%
\begingroup \smaller\smaller\smaller\begin{tabular}{@{}c@{}}%
8\\4\\0
\end{tabular}\endgroup%
{$\left.\llap{\phantom{%
\begingroup \smaller\smaller\smaller\begin{tabular}{@{}c@{}}%
0\\0\\0
\end{tabular}\endgroup%
}}\!\right]$}%
}%
\ifdim\wd\matricesbox>\halfwidth\myboxwidth=\hsize\else\myboxwidth=\halfwidth\fi
\vbox{%
\ifdim\myboxwidth=\hsize
\setbox\onelinebox=\hbox{%
\vbox{\hbox{%
$\Pi_{3,9}=\hbox{GN}_{13}$ spans $L_{7.9}$%
}\hbox{%
$\slashthree\infty|\infty\rtimes D_{2}$ (shared)%
}%
}%
\hfill\copy\matricesbox
}%
\ifdim\wd\onelinebox>\myboxwidth
\hbox to \myboxwidth{%
$\Pi_{3,9}=\hbox{GN}_{13}$ spans $L_{7.9}$%
\hfil
$\slashthree\infty|\infty\rtimes D_{2}$ (shared)%
}%
\box\matricesbox
\else
\hbox to \myboxwidth{%
\unhbox\onelinebox
}%
\fi
\else
\hbox to \myboxwidth{%
$\Pi_{3,9}=\hbox{GN}_{13}$ spans $L_{7.9}$%
\hfil}%
\hbox to \myboxwidth{%
$\slashthree\infty|\infty\rtimes D_{2}$ (shared)%
\hfil}%
\box\matricesbox
\fi
}%
\hfill\discretionary{}{}{}%
\setbox\matricesbox=\hbox{%
{$\left[\!\llap{\phantom{%
\begingroup \smaller\smaller\smaller\begin{tabular}{@{}c@{}}%
\phantom{0}\\\phantom{0}\\\phantom{0}
\end{tabular}\endgroup%
}}\right.$}%
\begingroup \smaller\smaller\smaller\begin{tabular}{@{}c@{}}%
-1/6\\\phantom{0}\\\phantom{0}
\end{tabular}\endgroup%
\kern3pt%
\begingroup \smaller\smaller\smaller\begin{tabular}{@{}c@{}}%
\phantom{0}\\1/6\\\phantom{0}
\end{tabular}\endgroup%
\kern3pt%
\begingroup \smaller\smaller\smaller\begin{tabular}{@{}c@{}}%
\phantom{0}\\\phantom{0}\\1
\end{tabular}\endgroup%
{$\left.\llap{\phantom{%
\begingroup \smaller\smaller\smaller\begin{tabular}{@{}c@{}}%
\phantom{0}\\\phantom{0}\\\phantom{0}
\end{tabular}\endgroup%
}}\!\right]$}%
{$\left[\!\llap{\phantom{%
\begingroup \smaller\smaller\smaller\begin{tabular}{@{}c@{}}%
0\\0\\0
\end{tabular}\endgroup%
}}\right.$}%
\begingroup \smaller\smaller\smaller\begin{tabular}{@{}c@{}}%
2\\-4\\0
\end{tabular}\endgroup%
\kern3pt%
\begingroup \smaller\smaller\smaller\begin{tabular}{@{}c@{}}%
1\\1\\-1
\end{tabular}\endgroup%
{$\left.\llap{\phantom{%
\begingroup \smaller\smaller\smaller\begin{tabular}{@{}c@{}}%
0\\0\\0
\end{tabular}\endgroup%
}}\!\right]$}%
}%
\ifdim\wd\matricesbox>\halfwidth\myboxwidth=\hsize\else\myboxwidth=\halfwidth\fi
\vbox{%
\ifdim\myboxwidth=\hsize
\setbox\onelinebox=\hbox{%
\vbox{\hbox{%
$\Pi_{3,10}$ spans $L_{1.6}$%
}\hbox{%
$4|4\slashinfty\rtimes D_{2}$ (shared)%
}%
}%
\hfill\copy\matricesbox
}%
\ifdim\wd\onelinebox>\myboxwidth
\hbox to \myboxwidth{%
$\Pi_{3,10}$ spans $L_{1.6}$%
\hfil
$4|4\slashinfty\rtimes D_{2}$ (shared)%
}%
\box\matricesbox
\else
\hbox to \myboxwidth{%
\unhbox\onelinebox
}%
\fi
\else
\hbox to \myboxwidth{%
$\Pi_{3,10}$ spans $L_{1.6}$%
\hfil}%
\hbox to \myboxwidth{%
$4|4\slashinfty\rtimes D_{2}$ (shared)%
\hfil}%
\box\matricesbox
\fi
}%
\hfill\discretionary{}{}{}%
\setbox\matricesbox=\hbox{%
{$\left[\!\llap{\phantom{%
\begingroup \smaller\smaller\smaller\begin{tabular}{@{}c@{}}%
\phantom{0}\\\phantom{0}\\\phantom{0}
\end{tabular}\endgroup%
}}\right.$}%
\begingroup \smaller\smaller\smaller\begin{tabular}{@{}c@{}}%
-1/8\\\phantom{0}\\\phantom{0}
\end{tabular}\endgroup%
\kern3pt%
\begingroup \smaller\smaller\smaller\begin{tabular}{@{}c@{}}%
\phantom{0}\\1/8\\\phantom{0}
\end{tabular}\endgroup%
\kern3pt%
\begingroup \smaller\smaller\smaller\begin{tabular}{@{}c@{}}%
\phantom{0}\\\phantom{0}\\2
\end{tabular}\endgroup%
{$\left.\llap{\phantom{%
\begingroup \smaller\smaller\smaller\begin{tabular}{@{}c@{}}%
\phantom{0}\\\phantom{0}\\\phantom{0}
\end{tabular}\endgroup%
}}\!\right]$}%
{$\left[\!\llap{\phantom{%
\begingroup \smaller\smaller\smaller\begin{tabular}{@{}c@{}}%
0\\0\\0
\end{tabular}\endgroup%
}}\right.$}%
\begingroup \smaller\smaller\smaller\begin{tabular}{@{}c@{}}%
1\\3\\0
\end{tabular}\endgroup%
\kern3pt%
\begingroup \smaller\smaller\smaller\begin{tabular}{@{}c@{}}%
2\\-2\\1
\end{tabular}\endgroup%
{$\left.\llap{\phantom{%
\begingroup \smaller\smaller\smaller\begin{tabular}{@{}c@{}}%
0\\0\\0
\end{tabular}\endgroup%
}}\!\right]$}%
}%
\ifdim\wd\matricesbox>\halfwidth\myboxwidth=\hsize\else\myboxwidth=\halfwidth\fi
\vbox{%
\ifdim\myboxwidth=\hsize
\setbox\onelinebox=\hbox{%
\vbox{\hbox{%
$\Pi_{3,11}$ spans $L_{1.3}$%
}\hbox{%
$4|4\slashinfty\rtimes D_{2}$ (shared)%
}%
}%
\hfill\copy\matricesbox
}%
\ifdim\wd\onelinebox>\myboxwidth
\hbox to \myboxwidth{%
$\Pi_{3,11}$ spans $L_{1.3}$%
\hfil
$4|4\slashinfty\rtimes D_{2}$ (shared)%
}%
\box\matricesbox
\else
\hbox to \myboxwidth{%
\unhbox\onelinebox
}%
\fi
\else
\hbox to \myboxwidth{%
$\Pi_{3,11}$ spans $L_{1.3}$%
\hfil}%
\hbox to \myboxwidth{%
$4|4\slashinfty\rtimes D_{2}$ (shared)%
\hfil}%
\box\matricesbox
\fi
}%
\hfill\discretionary{}{}{}%
\setbox\matricesbox=\hbox{%
{$\left[\!\llap{\phantom{%
\begingroup \smaller\smaller\smaller\begin{tabular}{@{}c@{}}%
\phantom{0}\\\phantom{0}\\\phantom{0}
\end{tabular}\endgroup%
}}\right.$}%
\begingroup \smaller\smaller\smaller\begin{tabular}{@{}c@{}}%
-3/32\\\phantom{0}\\\phantom{0}
\end{tabular}\endgroup%
\kern3pt%
\begingroup \smaller\smaller\smaller\begin{tabular}{@{}c@{}}%
\phantom{0}\\3/8\\\phantom{0}
\end{tabular}\endgroup%
\kern3pt%
\begingroup \smaller\smaller\smaller\begin{tabular}{@{}c@{}}%
\phantom{0}\\\phantom{0}\\2
\end{tabular}\endgroup%
{$\left.\llap{\phantom{%
\begingroup \smaller\smaller\smaller\begin{tabular}{@{}c@{}}%
\phantom{0}\\\phantom{0}\\\phantom{0}
\end{tabular}\endgroup%
}}\!\right]$}%
{$\left[\!\llap{\phantom{%
\begingroup \smaller\smaller\smaller\begin{tabular}{@{}c@{}}%
0\\0\\0
\end{tabular}\endgroup%
}}\right.$}%
\begingroup \smaller\smaller\smaller\begin{tabular}{@{}c@{}}%
6\\5\\0
\end{tabular}\endgroup%
\kern3pt%
\begingroup \smaller\smaller\smaller\begin{tabular}{@{}c@{}}%
2\\-1\\1
\end{tabular}\endgroup%
{$\left.\llap{\phantom{%
\begingroup \smaller\smaller\smaller\begin{tabular}{@{}c@{}}%
0\\0\\0
\end{tabular}\endgroup%
}}\!\right]$}%
}%
\ifdim\wd\matricesbox>\halfwidth\myboxwidth=\hsize\else\myboxwidth=\halfwidth\fi
\vbox{%
\ifdim\myboxwidth=\hsize
\setbox\onelinebox=\hbox{%
\vbox{\hbox{%
$\Pi_{3,12}$ spans $L_{7.9}$%
}\hbox{%
$6|6\slashinfty\rtimes D_{2}$ (shared)%
}%
}%
\hfill\copy\matricesbox
}%
\ifdim\wd\onelinebox>\myboxwidth
\hbox to \myboxwidth{%
$\Pi_{3,12}$ spans $L_{7.9}$%
\hfil
$6|6\slashinfty\rtimes D_{2}$ (shared)%
}%
\box\matricesbox
\else
\hbox to \myboxwidth{%
\unhbox\onelinebox
}%
\fi
\else
\hbox to \myboxwidth{%
$\Pi_{3,12}$ spans $L_{7.9}$%
\hfil}%
\hbox to \myboxwidth{%
$6|6\slashinfty\rtimes D_{2}$ (shared)%
\hfil}%
\box\matricesbox
\fi
}%
\hfill\discretionary{}{}{}%
\setbox\matricesbox=\hbox{%
{$\left[\!\llap{\phantom{%
\begingroup \smaller\smaller\smaller\begin{tabular}{@{}c@{}}%
\phantom{0}\\\phantom{0}\\\phantom{0}
\end{tabular}\endgroup%
}}\right.$}%
\begingroup \smaller\smaller\smaller\begin{tabular}{@{}c@{}}%
-1/16\\\phantom{0}\\\phantom{0}
\end{tabular}\endgroup%
\kern3pt%
\begingroup \smaller\smaller\smaller\begin{tabular}{@{}c@{}}%
\phantom{0}\\9/4\\\phantom{0}
\end{tabular}\endgroup%
\kern3pt%
\begingroup \smaller\smaller\smaller\begin{tabular}{@{}c@{}}%
\phantom{0}\\\phantom{0}\\6
\end{tabular}\endgroup%
{$\left.\llap{\phantom{%
\begingroup \smaller\smaller\smaller\begin{tabular}{@{}c@{}}%
\phantom{0}\\\phantom{0}\\\phantom{0}
\end{tabular}\endgroup%
}}\!\right]$}%
{$\left[\!\llap{\phantom{%
\begingroup \smaller\smaller\smaller\begin{tabular}{@{}c@{}}%
0\\0\\0
\end{tabular}\endgroup%
}}\right.$}%
\begingroup \smaller\smaller\smaller\begin{tabular}{@{}c@{}}%
2\\1\\0
\end{tabular}\endgroup%
\kern3pt%
\begingroup \smaller\smaller\smaller\begin{tabular}{@{}c@{}}%
6\\-1\\1
\end{tabular}\endgroup%
{$\left.\llap{\phantom{%
\begingroup \smaller\smaller\smaller\begin{tabular}{@{}c@{}}%
0\\0\\0
\end{tabular}\endgroup%
}}\!\right]$}%
}%
\ifdim\wd\matricesbox>\halfwidth\myboxwidth=\hsize\else\myboxwidth=\halfwidth\fi
\vbox{%
\ifdim\myboxwidth=\hsize
\setbox\onelinebox=\hbox{%
\vbox{\hbox{%
$\Pi_{3,13}$ spans $L_{7.13}$%
}\hbox{%
$6|6\slashinfty\rtimes D_{2}$ (shared)%
}%
}%
\hfill\copy\matricesbox
}%
\ifdim\wd\onelinebox>\myboxwidth
\hbox to \myboxwidth{%
$\Pi_{3,13}$ spans $L_{7.13}$%
\hfil
$6|6\slashinfty\rtimes D_{2}$ (shared)%
}%
\box\matricesbox
\else
\hbox to \myboxwidth{%
\unhbox\onelinebox
}%
\fi
\else
\hbox to \myboxwidth{%
$\Pi_{3,13}$ spans $L_{7.13}$%
\hfil}%
\hbox to \myboxwidth{%
$6|6\slashinfty\rtimes D_{2}$ (shared)%
\hfil}%
\box\matricesbox
\fi
}%
\hfill\discretionary{}{}{}%
\setbox\matricesbox=\hbox{%
{$\left[\!\llap{\phantom{%
\begingroup \smaller\smaller\smaller\begin{tabular}{@{}c@{}}%
\phantom{0}\\\phantom{0}\\\phantom{0}
\end{tabular}\endgroup%
}}\right.$}%
\begingroup \smaller\smaller\smaller\begin{tabular}{@{}c@{}}%
-1/5\\\phantom{0}\\\phantom{0}
\end{tabular}\endgroup%
\kern3pt%
\begingroup \smaller\smaller\smaller\begin{tabular}{@{}c@{}}%
\phantom{0}\\1/5\\\phantom{0}
\end{tabular}\endgroup%
\kern3pt%
\begingroup \smaller\smaller\smaller\begin{tabular}{@{}c@{}}%
\phantom{0}\\\phantom{0}\\1
\end{tabular}\endgroup%
{$\left.\llap{\phantom{%
\begingroup \smaller\smaller\smaller\begin{tabular}{@{}c@{}}%
\phantom{0}\\\phantom{0}\\\phantom{0}
\end{tabular}\endgroup%
}}\!\right]$}%
{$\left[\!\llap{\phantom{%
\begingroup \smaller\smaller\smaller\begin{tabular}{@{}c@{}}%
0\\0\\0
\end{tabular}\endgroup%
}}\right.$}%
\begingroup \smaller\smaller\smaller\begin{tabular}{@{}c@{}}%
4\\6\\0
\end{tabular}\endgroup%
\kern3pt%
\begingroup \smaller\smaller\smaller\begin{tabular}{@{}c@{}}%
1\\-1\\1
\end{tabular}\endgroup%
{$\left.\llap{\phantom{%
\begingroup \smaller\smaller\smaller\begin{tabular}{@{}c@{}}%
0\\0\\0
\end{tabular}\endgroup%
}}\!\right]$}%
}%
\ifdim\wd\matricesbox>\halfwidth\myboxwidth=\hsize\else\myboxwidth=\halfwidth\fi
\vbox{%
\ifdim\myboxwidth=\hsize
\setbox\onelinebox=\hbox{%
\vbox{\hbox{%
$\Pi_{3,14}=\hbox{GN}_{18}$ spans $L_{1.9}$%
}\hbox{%
$\infty|\infty\slashinfty\rtimes D_{2}$ (shared)%
}%
}%
\hfill\copy\matricesbox
}%
\ifdim\wd\onelinebox>\myboxwidth
\hbox to \myboxwidth{%
$\Pi_{3,14}=\hbox{GN}_{18}$ spans $L_{1.9}$%
\hfil
$\infty|\infty\slashinfty\rtimes D_{2}$ (shared)%
}%
\box\matricesbox
\else
\hbox to \myboxwidth{%
\unhbox\onelinebox
}%
\fi
\else
\hbox to \myboxwidth{%
$\Pi_{3,14}=\hbox{GN}_{18}$ spans $L_{1.9}$%
\hfil}%
\hbox to \myboxwidth{%
$\infty|\infty\slashinfty\rtimes D_{2}$ (shared)%
\hfil}%
\box\matricesbox
\fi
}%
\hfill\discretionary{}{}{}%
\setbox\matricesbox=\hbox{%
{$\left[\!\llap{\phantom{%
\begingroup \smaller\smaller\smaller\begin{tabular}{@{}c@{}}%
\phantom{0}\\\phantom{0}\\\phantom{0}
\end{tabular}\endgroup%
}}\right.$}%
\begingroup \smaller\smaller\smaller\begin{tabular}{@{}c@{}}%
-1/8\\\phantom{0}\\\phantom{0}
\end{tabular}\endgroup%
\kern3pt%
\begingroup \smaller\smaller\smaller\begin{tabular}{@{}c@{}}%
\phantom{0}\\1/8\\\phantom{0}
\end{tabular}\endgroup%
\kern3pt%
\begingroup \smaller\smaller\smaller\begin{tabular}{@{}c@{}}%
\phantom{0}\\\phantom{0}\\4
\end{tabular}\endgroup%
{$\left.\llap{\phantom{%
\begingroup \smaller\smaller\smaller\begin{tabular}{@{}c@{}}%
\phantom{0}\\\phantom{0}\\\phantom{0}
\end{tabular}\endgroup%
}}\!\right]$}%
{$\left[\!\llap{\phantom{%
\begingroup \smaller\smaller\smaller\begin{tabular}{@{}c@{}}%
0\\0\\0
\end{tabular}\endgroup%
}}\right.$}%
\begingroup \smaller\smaller\smaller\begin{tabular}{@{}c@{}}%
1\\-3\\0
\end{tabular}\endgroup%
\kern3pt%
\begingroup \smaller\smaller\smaller\begin{tabular}{@{}c@{}}%
4\\4\\1
\end{tabular}\endgroup%
{$\left.\llap{\phantom{%
\begingroup \smaller\smaller\smaller\begin{tabular}{@{}c@{}}%
0\\0\\0
\end{tabular}\endgroup%
}}\!\right]$}%
}%
\ifdim\wd\matricesbox>\halfwidth\myboxwidth=\hsize\else\myboxwidth=\halfwidth\fi
\vbox{%
\ifdim\myboxwidth=\hsize
\setbox\onelinebox=\hbox{%
\vbox{\hbox{%
$\Pi_{3,15}=A_{4,I,\zerobar}=\hbox{GN}_{14}$ spans $L_{140.4}$%
}\hbox{%
$|\infty\slashinfty\infty\rtimes D_{2}$ (shared)%
}%
}%
\hfill\copy\matricesbox
}%
\ifdim\wd\onelinebox>\myboxwidth
\hbox to \myboxwidth{%
$\Pi_{3,15}=A_{4,I,\zerobar}=\hbox{GN}_{14}$ spans $L_{140.4}$%
\hfil
$|\infty\slashinfty\infty\rtimes D_{2}$ (shared)%
}%
\box\matricesbox
\else
\hbox to \myboxwidth{%
\unhbox\onelinebox
}%
\fi
\else
\hbox to \myboxwidth{%
$\Pi_{3,15}=A_{4,I,\zerobar}=\hbox{GN}_{14}$ spans $L_{140.4}$%
\hfil}%
\hbox to \myboxwidth{%
$|\infty\slashinfty\infty\rtimes D_{2}$ (shared)%
\hfil}%
\box\matricesbox
\fi
}%
\hfill\discretionary{}{}{}%
\setbox\matricesbox=\hbox{%
{$\left[\!\llap{\phantom{%
\begingroup \smaller\smaller\smaller\begin{tabular}{@{}c@{}}%
\phantom{0}\\\phantom{0}\\\phantom{0}
\end{tabular}\endgroup%
}}\right.$}%
\begingroup \smaller\smaller\smaller\begin{tabular}{@{}c@{}}%
-1/52\\\phantom{0}\\\phantom{0}
\end{tabular}\endgroup%
\kern3pt%
\begingroup \smaller\smaller\smaller\begin{tabular}{@{}c@{}}%
\phantom{0}\\3/13\\\phantom{0}
\end{tabular}\endgroup%
\kern3pt%
\begingroup \smaller\smaller\smaller\begin{tabular}{@{}c@{}}%
\phantom{0}\\\phantom{0}\\3
\end{tabular}\endgroup%
{$\left.\llap{\phantom{%
\begingroup \smaller\smaller\smaller\begin{tabular}{@{}c@{}}%
\phantom{0}\\\phantom{0}\\\phantom{0}
\end{tabular}\endgroup%
}}\!\right]$}%
{$\left[\!\llap{\phantom{%
\begingroup \smaller\smaller\smaller\begin{tabular}{@{}c@{}}%
0\\0\\0
\end{tabular}\endgroup%
}}\right.$}%
\begingroup \smaller\smaller\smaller\begin{tabular}{@{}c@{}}%
2\\-3\\0
\end{tabular}\endgroup%
\kern3pt%
\begingroup \smaller\smaller\smaller\begin{tabular}{@{}c@{}}%
6\\4\\1
\end{tabular}\endgroup%
{$\left.\llap{\phantom{%
\begingroup \smaller\smaller\smaller\begin{tabular}{@{}c@{}}%
0\\0\\0
\end{tabular}\endgroup%
}}\!\right]$}%
}%
\ifdim\wd\matricesbox>\halfwidth\myboxwidth=\hsize\else\myboxwidth=\halfwidth\fi
\vbox{%
\ifdim\myboxwidth=\hsize
\setbox\onelinebox=\hbox{%
\vbox{\hbox{%
$\Pi_{3,16}$ spans $L_{3.4}$%
}\hbox{%
$|6\slashtwo6\rtimes D_{2}$ (shared)%
}%
}%
\hfill\copy\matricesbox
}%
\ifdim\wd\onelinebox>\myboxwidth
\hbox to \myboxwidth{%
$\Pi_{3,16}$ spans $L_{3.4}$%
\hfil
$|6\slashtwo6\rtimes D_{2}$ (shared)%
}%
\box\matricesbox
\else
\hbox to \myboxwidth{%
\unhbox\onelinebox
}%
\fi
\else
\hbox to \myboxwidth{%
$\Pi_{3,16}$ spans $L_{3.4}$%
\hfil}%
\hbox to \myboxwidth{%
$|6\slashtwo6\rtimes D_{2}$ (shared)%
\hfil}%
\box\matricesbox
\fi
}%
\hfill\discretionary{}{}{}%
\setbox\matricesbox=\hbox{%
{$\left[\!\llap{\phantom{%
\begingroup \smaller\smaller\smaller\begin{tabular}{@{}c@{}}%
\phantom{0}\\\phantom{0}\\\phantom{0}
\end{tabular}\endgroup%
}}\right.$}%
\begingroup \smaller\smaller\smaller\begin{tabular}{@{}c@{}}%
-1/17\\\phantom{0}\\\phantom{0}
\end{tabular}\endgroup%
\kern3pt%
\begingroup \smaller\smaller\smaller\begin{tabular}{@{}c@{}}%
\phantom{0}\\2/17\\\phantom{0}
\end{tabular}\endgroup%
\kern3pt%
\begingroup \smaller\smaller\smaller\begin{tabular}{@{}c@{}}%
\phantom{0}\\\phantom{0}\\2
\end{tabular}\endgroup%
{$\left.\llap{\phantom{%
\begingroup \smaller\smaller\smaller\begin{tabular}{@{}c@{}}%
\phantom{0}\\\phantom{0}\\\phantom{0}
\end{tabular}\endgroup%
}}\!\right]$}%
{$\left[\!\llap{\phantom{%
\begingroup \smaller\smaller\smaller\begin{tabular}{@{}c@{}}%
0\\0\\0
\end{tabular}\endgroup%
}}\right.$}%
\begingroup \smaller\smaller\smaller\begin{tabular}{@{}c@{}}%
1\\-3\\0
\end{tabular}\endgroup%
\kern3pt%
\begingroup \smaller\smaller\smaller\begin{tabular}{@{}c@{}}%
4\\5\\1
\end{tabular}\endgroup%
{$\left.\llap{\phantom{%
\begingroup \smaller\smaller\smaller\begin{tabular}{@{}c@{}}%
0\\0\\0
\end{tabular}\endgroup%
}}\!\right]$}%
}%
\ifdim\wd\matricesbox>\halfwidth\myboxwidth=\hsize\else\myboxwidth=\halfwidth\fi
\vbox{%
\ifdim\myboxwidth=\hsize
\setbox\onelinebox=\hbox{%
\vbox{\hbox{%
$\Pi_{3,17}=A_{2,I,\zerobar}=\hbox{GN}_{6}$ spans $L_{1.9}$%
}\hbox{%
$|\infty\slashtwo\infty\rtimes D_{2}$ (shared)%
}%
}%
\hfill\copy\matricesbox
}%
\ifdim\wd\onelinebox>\myboxwidth
\hbox to \myboxwidth{%
$\Pi_{3,17}=A_{2,I,\zerobar}=\hbox{GN}_{6}$ spans $L_{1.9}$%
\hfil
$|\infty\slashtwo\infty\rtimes D_{2}$ (shared)%
}%
\box\matricesbox
\else
\hbox to \myboxwidth{%
\unhbox\onelinebox
}%
\fi
\else
\hbox to \myboxwidth{%
$\Pi_{3,17}=A_{2,I,\zerobar}=\hbox{GN}_{6}$ spans $L_{1.9}$%
\hfil}%
\hbox to \myboxwidth{%
$|\infty\slashtwo\infty\rtimes D_{2}$ (shared)%
\hfil}%
\box\matricesbox
\fi
}%
\hfill\discretionary{}{}{}%
\setbox\matricesbox=\hbox{%
{$\left[\!\llap{\phantom{%
\begingroup \smaller\smaller\smaller\begin{tabular}{@{}c@{}}%
\phantom{0}\\\phantom{0}\\\phantom{0}
\end{tabular}\endgroup%
}}\right.$}%
\begingroup \smaller\smaller\smaller\begin{tabular}{@{}c@{}}%
-1/7\\\phantom{0}\\\phantom{0}
\end{tabular}\endgroup%
\kern3pt%
\begingroup \smaller\smaller\smaller\begin{tabular}{@{}c@{}}%
\phantom{0}\\1/14\\\phantom{0}
\end{tabular}\endgroup%
\kern3pt%
\begingroup \smaller\smaller\smaller\begin{tabular}{@{}c@{}}%
\phantom{0}\\\phantom{0}\\1/2
\end{tabular}\endgroup%
{$\left.\llap{\phantom{%
\begingroup \smaller\smaller\smaller\begin{tabular}{@{}c@{}}%
\phantom{0}\\\phantom{0}\\\phantom{0}
\end{tabular}\endgroup%
}}\!\right]$}%
{$\left[\!\llap{\phantom{%
\begingroup \smaller\smaller\smaller\begin{tabular}{@{}c@{}}%
0\\0\\0
\end{tabular}\endgroup%
}}\right.$}%
\begingroup \smaller\smaller\smaller\begin{tabular}{@{}c@{}}%
1\\4\\0
\end{tabular}\endgroup%
\kern3pt%
\begingroup \smaller\smaller\smaller\begin{tabular}{@{}c@{}}%
1\\-3\\-1
\end{tabular}\endgroup%
{$\left.\llap{\phantom{%
\begingroup \smaller\smaller\smaller\begin{tabular}{@{}c@{}}%
0\\0\\0
\end{tabular}\endgroup%
}}\!\right]$}%
}%
\ifdim\wd\matricesbox>\halfwidth\myboxwidth=\hsize\else\myboxwidth=\halfwidth\fi
\vbox{%
\ifdim\myboxwidth=\hsize
\setbox\onelinebox=\hbox{%
\vbox{\hbox{%
$\Pi_{3,18}=A_{2,0}=\hbox{GN}_{17}$ spans $L_{1.4}$%
}\hbox{%
$|\infty\slashtwo\infty\rtimes D_{2}$ (shared)%
}%
}%
\hfill\copy\matricesbox
}%
\ifdim\wd\onelinebox>\myboxwidth
\hbox to \myboxwidth{%
$\Pi_{3,18}=A_{2,0}=\hbox{GN}_{17}$ spans $L_{1.4}$%
\hfil
$|\infty\slashtwo\infty\rtimes D_{2}$ (shared)%
}%
\box\matricesbox
\else
\hbox to \myboxwidth{%
\unhbox\onelinebox
}%
\fi
\else
\hbox to \myboxwidth{%
$\Pi_{3,18}=A_{2,0}=\hbox{GN}_{17}$ spans $L_{1.4}$%
\hfil}%
\hbox to \myboxwidth{%
$|\infty\slashtwo\infty\rtimes D_{2}$ (shared)%
\hfil}%
\box\matricesbox
\fi
}%
\hfill\discretionary{}{}{}%
\setbox\matricesbox=\hbox{%
{$\left[\!\llap{\phantom{%
\begingroup \smaller\smaller\smaller\begin{tabular}{@{}c@{}}%
\phantom{0}\\\phantom{0}\\\phantom{0}
\end{tabular}\endgroup%
}}\right.$}%
\begingroup \smaller\smaller\smaller\begin{tabular}{@{}c@{}}%
-1/44\\\phantom{0}\\\phantom{0}
\end{tabular}\endgroup%
\kern3pt%
\begingroup \smaller\smaller\smaller\begin{tabular}{@{}c@{}}%
\phantom{0}\\3/11\\\phantom{0}
\end{tabular}\endgroup%
\kern3pt%
\begingroup \smaller\smaller\smaller\begin{tabular}{@{}c@{}}%
\phantom{0}\\\phantom{0}\\1
\end{tabular}\endgroup%
{$\left.\llap{\phantom{%
\begingroup \smaller\smaller\smaller\begin{tabular}{@{}c@{}}%
\phantom{0}\\\phantom{0}\\\phantom{0}
\end{tabular}\endgroup%
}}\!\right]$}%
{$\left[\!\llap{\phantom{%
\begingroup \smaller\smaller\smaller\begin{tabular}{@{}c@{}}%
0\\0\\0
\end{tabular}\endgroup%
}}\right.$}%
\begingroup \smaller\smaller\smaller\begin{tabular}{@{}c@{}}%
6\\-5\\0
\end{tabular}\endgroup%
\kern3pt%
\begingroup \smaller\smaller\smaller\begin{tabular}{@{}c@{}}%
2\\2\\-1
\end{tabular}\endgroup%
{$\left.\llap{\phantom{%
\begingroup \smaller\smaller\smaller\begin{tabular}{@{}c@{}}%
0\\0\\0
\end{tabular}\endgroup%
}}\!\right]$}%
}%
\ifdim\wd\matricesbox>\halfwidth\myboxwidth=\hsize\else\myboxwidth=\halfwidth\fi
\vbox{%
\ifdim\myboxwidth=\hsize
\setbox\onelinebox=\hbox{%
\vbox{\hbox{%
$\Pi_{3,19}$ spans $L_{3.1}$%
}\hbox{%
$|6\slashtwo6\rtimes D_{2}$ (shared)%
}%
}%
\hfill\copy\matricesbox
}%
\ifdim\wd\onelinebox>\myboxwidth
\hbox to \myboxwidth{%
$\Pi_{3,19}$ spans $L_{3.1}$%
\hfil
$|6\slashtwo6\rtimes D_{2}$ (shared)%
}%
\box\matricesbox
\else
\hbox to \myboxwidth{%
\unhbox\onelinebox
}%
\fi
\else
\hbox to \myboxwidth{%
$\Pi_{3,19}$ spans $L_{3.1}$%
\hfil}%
\hbox to \myboxwidth{%
$|6\slashtwo6\rtimes D_{2}$ (shared)%
\hfil}%
\box\matricesbox
\fi
}%
\hfill\discretionary{}{}{}%
\setbox\matricesbox=\hbox{%
{$\left[\!\llap{\phantom{%
\begingroup \smaller\smaller\smaller\begin{tabular}{@{}c@{}}%
\phantom{0}\\\phantom{0}\\\phantom{0}
\end{tabular}\endgroup%
}}\right.$}%
\begingroup \smaller\smaller\smaller\begin{tabular}{@{}c@{}}%
-1/14\\\phantom{0}\\\phantom{0}
\end{tabular}\endgroup%
\kern3pt%
\begingroup \smaller\smaller\smaller\begin{tabular}{@{}c@{}}%
\phantom{0}\\4/7\\\phantom{0}
\end{tabular}\endgroup%
\kern3pt%
\begingroup \smaller\smaller\smaller\begin{tabular}{@{}c@{}}%
\phantom{0}\\\phantom{0}\\1/2
\end{tabular}\endgroup%
{$\left.\llap{\phantom{%
\begingroup \smaller\smaller\smaller\begin{tabular}{@{}c@{}}%
\phantom{0}\\\phantom{0}\\\phantom{0}
\end{tabular}\endgroup%
}}\!\right]$}%
{$\left[\!\llap{\phantom{%
\begingroup \smaller\smaller\smaller\begin{tabular}{@{}c@{}}%
0\\0\\0
\end{tabular}\endgroup%
}}\right.$}%
\begingroup \smaller\smaller\smaller\begin{tabular}{@{}c@{}}%
4\\3\\0
\end{tabular}\endgroup%
\kern3pt%
\begingroup \smaller\smaller\smaller\begin{tabular}{@{}c@{}}%
1\\-1\\-1
\end{tabular}\endgroup%
{$\left.\llap{\phantom{%
\begingroup \smaller\smaller\smaller\begin{tabular}{@{}c@{}}%
0\\0\\0
\end{tabular}\endgroup%
}}\!\right]$}%
}%
\ifdim\wd\matricesbox>\halfwidth\myboxwidth=\hsize\else\myboxwidth=\halfwidth\fi
\vbox{%
\ifdim\myboxwidth=\hsize
\setbox\onelinebox=\hbox{%
\vbox{\hbox{%
$\Pi_{3,20}=\hbox{GN}_{7}$ spans $L_{1.6}$%
}\hbox{%
$|\infty\slashtwo\infty\rtimes D_{2}$ (shared)%
}%
}%
\hfill\copy\matricesbox
}%
\ifdim\wd\onelinebox>\myboxwidth
\hbox to \myboxwidth{%
$\Pi_{3,20}=\hbox{GN}_{7}$ spans $L_{1.6}$%
\hfil
$|\infty\slashtwo\infty\rtimes D_{2}$ (shared)%
}%
\box\matricesbox
\else
\hbox to \myboxwidth{%
\unhbox\onelinebox
}%
\fi
\else
\hbox to \myboxwidth{%
$\Pi_{3,20}=\hbox{GN}_{7}$ spans $L_{1.6}$%
\hfil}%
\hbox to \myboxwidth{%
$|\infty\slashtwo\infty\rtimes D_{2}$ (shared)%
\hfil}%
\box\matricesbox
\fi
}%
\hfill\discretionary{}{}{}%
\setbox\matricesbox=\hbox{%
{$\left[\!\llap{\phantom{%
\begingroup \smaller\smaller\smaller\begin{tabular}{@{}c@{}}%
\phantom{0}\\\phantom{0}\\\phantom{0}
\end{tabular}\endgroup%
}}\right.$}%
\begingroup \smaller\smaller\smaller\begin{tabular}{@{}c@{}}%
-4/41\\\phantom{0}\\\phantom{0}
\end{tabular}\endgroup%
\kern3pt%
\begingroup \smaller\smaller\smaller\begin{tabular}{@{}c@{}}%
\phantom{0}\\5/41\\-2/41
\end{tabular}\endgroup%
\kern3pt%
\begingroup \smaller\smaller\smaller\begin{tabular}{@{}c@{}}%
\phantom{0}\\-2/41\\9/41
\end{tabular}\endgroup%
{$\left.\llap{\phantom{%
\begingroup \smaller\smaller\smaller\begin{tabular}{@{}c@{}}%
\phantom{0}\\\phantom{0}\\\phantom{0}
\end{tabular}\endgroup%
}}\!\right]$}%
{$\left[\!\llap{\phantom{%
\begingroup \smaller\smaller\smaller\begin{tabular}{@{}c@{}}%
0\\0\\0
\end{tabular}\endgroup%
}}\right.$}%
\begingroup \smaller\smaller\smaller\begin{tabular}{@{}c@{}}%
1\\3\\0
\end{tabular}\endgroup%
\kern3pt%
\begingroup \smaller\smaller\smaller\begin{tabular}{@{}c@{}}%
2\\-1\\3
\end{tabular}\endgroup%
\kern3pt%
\begingroup \smaller\smaller\smaller\begin{tabular}{@{}c@{}}%
4\\-6\\-4
\end{tabular}\endgroup%
{$\left.\llap{\phantom{%
\begingroup \smaller\smaller\smaller\begin{tabular}{@{}c@{}}%
0\\0\\0
\end{tabular}\endgroup%
}}\!\right]$}%
}%
\ifdim\wd\matricesbox>\halfwidth\myboxwidth=\hsize\else\myboxwidth=\halfwidth\fi
\vbox{%
\ifdim\myboxwidth=\hsize
\setbox\onelinebox=\hbox{%
\vbox{\hbox{%
$\Pi_{3,21}$ spans $L_{145.1}$%
}\hbox{%
$44\infty$ (shared)%
}%
}%
\hfill\copy\matricesbox
}%
\ifdim\wd\onelinebox>\myboxwidth
\hbox to \myboxwidth{%
$\Pi_{3,21}$ spans $L_{145.1}$%
\hfil
$44\infty$ (shared)%
}%
\box\matricesbox
\else
\hbox to \myboxwidth{%
\unhbox\onelinebox
}%
\fi
\else
\hbox to \myboxwidth{%
$\Pi_{3,21}$ spans $L_{145.1}$%
\hfil}%
\hbox to \myboxwidth{%
$44\infty$ (shared)%
\hfil}%
\box\matricesbox
\fi
}%
\hfill\discretionary{}{}{}%
\setbox\matricesbox=\hbox{%
{$\left[\!\llap{\phantom{%
\begingroup \smaller\smaller\smaller\begin{tabular}{@{}c@{}}%
\phantom{0}\\\phantom{0}\\\phantom{0}
\end{tabular}\endgroup%
}}\right.$}%
\begingroup \smaller\smaller\smaller\begin{tabular}{@{}c@{}}%
-1/88\\\phantom{0}\\\phantom{0}
\end{tabular}\endgroup%
\kern3pt%
\begingroup \smaller\smaller\smaller\begin{tabular}{@{}c@{}}%
\phantom{0}\\5/22\\-1/11
\end{tabular}\endgroup%
\kern3pt%
\begingroup \smaller\smaller\smaller\begin{tabular}{@{}c@{}}%
\phantom{0}\\-1/11\\18/11
\end{tabular}\endgroup%
{$\left.\llap{\phantom{%
\begingroup \smaller\smaller\smaller\begin{tabular}{@{}c@{}}%
\phantom{0}\\\phantom{0}\\\phantom{0}
\end{tabular}\endgroup%
}}\!\right]$}%
{$\left[\!\llap{\phantom{%
\begingroup \smaller\smaller\smaller\begin{tabular}{@{}c@{}}%
0\\0\\0
\end{tabular}\endgroup%
}}\right.$}%
\begingroup \smaller\smaller\smaller\begin{tabular}{@{}c@{}}%
2\\3\\0
\end{tabular}\endgroup%
\kern3pt%
\begingroup \smaller\smaller\smaller\begin{tabular}{@{}c@{}}%
2\\-1\\1
\end{tabular}\endgroup%
\kern3pt%
\begingroup \smaller\smaller\smaller\begin{tabular}{@{}c@{}}%
8\\-6\\-1
\end{tabular}\endgroup%
{$\left.\llap{\phantom{%
\begingroup \smaller\smaller\smaller\begin{tabular}{@{}c@{}}%
0\\0\\0
\end{tabular}\endgroup%
}}\!\right]$}%
}%
\ifdim\wd\matricesbox>\halfwidth\myboxwidth=\hsize\else\myboxwidth=\halfwidth\fi
\vbox{%
\ifdim\myboxwidth=\hsize
\setbox\onelinebox=\hbox{%
\vbox{\hbox{%
$\Pi_{3,22}=\hbox{GN}_{2}$ spans $L_{2.1}$%
}\hbox{%
$32\infty$ (shared)%
}%
}%
\hfill\copy\matricesbox
}%
\ifdim\wd\onelinebox>\myboxwidth
\hbox to \myboxwidth{%
$\Pi_{3,22}=\hbox{GN}_{2}$ spans $L_{2.1}$%
\hfil
$32\infty$ (shared)%
}%
\box\matricesbox
\else
\hbox to \myboxwidth{%
\unhbox\onelinebox
}%
\fi
\else
\hbox to \myboxwidth{%
$\Pi_{3,22}=\hbox{GN}_{2}$ spans $L_{2.1}$%
\hfil}%
\hbox to \myboxwidth{%
$32\infty$ (shared)%
\hfil}%
\box\matricesbox
\fi
}%
\hfill\discretionary{}{}{}%
\setbox\matricesbox=\hbox{%
{$\left[\!\llap{\phantom{%
\begingroup \smaller\smaller\smaller\begin{tabular}{@{}c@{}}%
\phantom{0}\\\phantom{0}\\\phantom{0}
\end{tabular}\endgroup%
}}\right.$}%
\begingroup \smaller\smaller\smaller\begin{tabular}{@{}c@{}}%
-1/46\\\phantom{0}\\\phantom{0}
\end{tabular}\endgroup%
\kern3pt%
\begingroup \smaller\smaller\smaller\begin{tabular}{@{}c@{}}%
\phantom{0}\\4/23\\-1/23
\end{tabular}\endgroup%
\kern3pt%
\begingroup \smaller\smaller\smaller\begin{tabular}{@{}c@{}}%
\phantom{0}\\-1/23\\6/23
\end{tabular}\endgroup%
{$\left.\llap{\phantom{%
\begingroup \smaller\smaller\smaller\begin{tabular}{@{}c@{}}%
\phantom{0}\\\phantom{0}\\\phantom{0}
\end{tabular}\endgroup%
}}\!\right]$}%
{$\left[\!\llap{\phantom{%
\begingroup \smaller\smaller\smaller\begin{tabular}{@{}c@{}}%
0\\0\\0
\end{tabular}\endgroup%
}}\right.$}%
\begingroup \smaller\smaller\smaller\begin{tabular}{@{}c@{}}%
2\\3\\-1
\end{tabular}\endgroup%
\kern3pt%
\begingroup \smaller\smaller\smaller\begin{tabular}{@{}c@{}}%
2\\-3\\-2
\end{tabular}\endgroup%
\kern3pt%
\begingroup \smaller\smaller\smaller\begin{tabular}{@{}c@{}}%
2\\-2\\2
\end{tabular}\endgroup%
{$\left.\llap{\phantom{%
\begingroup \smaller\smaller\smaller\begin{tabular}{@{}c@{}}%
0\\0\\0
\end{tabular}\endgroup%
}}\!\right]$}%
}%
\ifdim\wd\matricesbox>\halfwidth\myboxwidth=\hsize\else\myboxwidth=\halfwidth\fi
\vbox{%
\ifdim\myboxwidth=\hsize
\setbox\onelinebox=\hbox{%
\vbox{\hbox{%
$\Pi_{3,23}=A_{1,0}=\hbox{GN}_{4}$ spans $L_{2.2}$%
}\hbox{%
$32\infty$ (shared)%
}%
}%
\hfill\copy\matricesbox
}%
\ifdim\wd\onelinebox>\myboxwidth
\hbox to \myboxwidth{%
$\Pi_{3,23}=A_{1,0}=\hbox{GN}_{4}$ spans $L_{2.2}$%
\hfil
$32\infty$ (shared)%
}%
\box\matricesbox
\else
\hbox to \myboxwidth{%
\unhbox\onelinebox
}%
\fi
\else
\hbox to \myboxwidth{%
$\Pi_{3,23}=A_{1,0}=\hbox{GN}_{4}$ spans $L_{2.2}$%
\hfil}%
\hbox to \myboxwidth{%
$32\infty$ (shared)%
\hfil}%
\box\matricesbox
\fi
}%
\hfill\discretionary{}{}{}%
\setbox\matricesbox=\hbox{%
{$\left[\!\llap{\phantom{%
\begingroup \smaller\smaller\smaller\begin{tabular}{@{}c@{}}%
\phantom{0}\\\phantom{0}\\\phantom{0}
\end{tabular}\endgroup%
}}\right.$}%
\begingroup \smaller\smaller\smaller\begin{tabular}{@{}c@{}}%
-1/59\\\phantom{0}\\\phantom{0}
\end{tabular}\endgroup%
\kern3pt%
\begingroup \smaller\smaller\smaller\begin{tabular}{@{}c@{}}%
\phantom{0}\\12/59\\-2/59
\end{tabular}\endgroup%
\kern3pt%
\begingroup \smaller\smaller\smaller\begin{tabular}{@{}c@{}}%
\phantom{0}\\-2/59\\20/59
\end{tabular}\endgroup%
{$\left.\llap{\phantom{%
\begingroup \smaller\smaller\smaller\begin{tabular}{@{}c@{}}%
\phantom{0}\\\phantom{0}\\\phantom{0}
\end{tabular}\endgroup%
}}\!\right]$}%
{$\left[\!\llap{\phantom{%
\begingroup \smaller\smaller\smaller\begin{tabular}{@{}c@{}}%
0\\0\\0
\end{tabular}\endgroup%
}}\right.$}%
\begingroup \smaller\smaller\smaller\begin{tabular}{@{}c@{}}%
4\\3\\3
\end{tabular}\endgroup%
\kern3pt%
\begingroup \smaller\smaller\smaller\begin{tabular}{@{}c@{}}%
4\\2\\-3
\end{tabular}\endgroup%
\kern3pt%
\begingroup \smaller\smaller\smaller\begin{tabular}{@{}c@{}}%
1\\-2\\-1
\end{tabular}\endgroup%
{$\left.\llap{\phantom{%
\begingroup \smaller\smaller\smaller\begin{tabular}{@{}c@{}}%
0\\0\\0
\end{tabular}\endgroup%
}}\!\right]$}%
}%
\ifdim\wd\matricesbox>\halfwidth\myboxwidth=\hsize\else\myboxwidth=\halfwidth\fi
\vbox{%
\ifdim\myboxwidth=\hsize
\setbox\onelinebox=\hbox{%
\vbox{\hbox{%
$\Pi_{3,24}=A_{1,I,\zerobar}=\hbox{GN}_{1}$ spans $L_{2.3}$%
}\hbox{%
$32\infty$ (shared)%
}%
}%
\hfill\copy\matricesbox
}%
\ifdim\wd\onelinebox>\myboxwidth
\hbox to \myboxwidth{%
$\Pi_{3,24}=A_{1,I,\zerobar}=\hbox{GN}_{1}$ spans $L_{2.3}$%
\hfil
$32\infty$ (shared)%
}%
\box\matricesbox
\else
\hbox to \myboxwidth{%
\unhbox\onelinebox
}%
\fi
\else
\hbox to \myboxwidth{%
$\Pi_{3,24}=A_{1,I,\zerobar}=\hbox{GN}_{1}$ spans $L_{2.3}$%
\hfil}%
\hbox to \myboxwidth{%
$32\infty$ (shared)%
\hfil}%
\box\matricesbox
\fi
}%
\hfill\discretionary{}{}{}%
\setbox\matricesbox=\hbox{%
{$\left[\!\llap{\phantom{%
\begingroup \smaller\smaller\smaller\begin{tabular}{@{}c@{}}%
\phantom{0}\\\phantom{0}\\\phantom{0}
\end{tabular}\endgroup%
}}\right.$}%
\begingroup \smaller\smaller\smaller\begin{tabular}{@{}c@{}}%
-1/88\\\phantom{0}\\\phantom{0}
\end{tabular}\endgroup%
\kern3pt%
\begingroup \smaller\smaller\smaller\begin{tabular}{@{}c@{}}%
\phantom{0}\\5/22\\-1/22
\end{tabular}\endgroup%
\kern3pt%
\begingroup \smaller\smaller\smaller\begin{tabular}{@{}c@{}}%
\phantom{0}\\-1/22\\53/22
\end{tabular}\endgroup%
{$\left.\llap{\phantom{%
\begingroup \smaller\smaller\smaller\begin{tabular}{@{}c@{}}%
\phantom{0}\\\phantom{0}\\\phantom{0}
\end{tabular}\endgroup%
}}\!\right]$}%
{$\left[\!\llap{\phantom{%
\begingroup \smaller\smaller\smaller\begin{tabular}{@{}c@{}}%
0\\0\\0
\end{tabular}\endgroup%
}}\right.$}%
\begingroup \smaller\smaller\smaller\begin{tabular}{@{}c@{}}%
2\\3\\0
\end{tabular}\endgroup%
\kern3pt%
\begingroup \smaller\smaller\smaller\begin{tabular}{@{}c@{}}%
4\\-3\\-1
\end{tabular}\endgroup%
\kern3pt%
\begingroup \smaller\smaller\smaller\begin{tabular}{@{}c@{}}%
6\\-4\\1
\end{tabular}\endgroup%
{$\left.\llap{\phantom{%
\begingroup \smaller\smaller\smaller\begin{tabular}{@{}c@{}}%
0\\0\\0
\end{tabular}\endgroup%
}}\!\right]$}%
}%
\ifdim\wd\matricesbox>\halfwidth\myboxwidth=\hsize\else\myboxwidth=\halfwidth\fi
\vbox{%
\ifdim\myboxwidth=\hsize
\setbox\onelinebox=\hbox{%
\vbox{\hbox{%
$\Pi_{3,25}$ spans $L_{3.1}$%
}\hbox{%
$426$ (shared)%
}%
}%
\hfill\copy\matricesbox
}%
\ifdim\wd\onelinebox>\myboxwidth
\hbox to \myboxwidth{%
$\Pi_{3,25}$ spans $L_{3.1}$%
\hfil
$426$ (shared)%
}%
\box\matricesbox
\else
\hbox to \myboxwidth{%
\unhbox\onelinebox
}%
\fi
\else
\hbox to \myboxwidth{%
$\Pi_{3,25}$ spans $L_{3.1}$%
\hfil}%
\hbox to \myboxwidth{%
$426$ (shared)%
\hfil}%
\box\matricesbox
\fi
}%
\hfill\discretionary{}{}{}%
\setbox\matricesbox=\hbox{%
{$\left[\!\llap{\phantom{%
\begingroup \smaller\smaller\smaller\begin{tabular}{@{}c@{}}%
\phantom{0}\\\phantom{0}\\\phantom{0}
\end{tabular}\endgroup%
}}\right.$}%
\begingroup \smaller\smaller\smaller\begin{tabular}{@{}c@{}}%
-1/26\\\phantom{0}\\\phantom{0}
\end{tabular}\endgroup%
\kern3pt%
\begingroup \smaller\smaller\smaller\begin{tabular}{@{}c@{}}%
\phantom{0}\\3/26\\-1/26
\end{tabular}\endgroup%
\kern3pt%
\begingroup \smaller\smaller\smaller\begin{tabular}{@{}c@{}}%
\phantom{0}\\-1/26\\35/26
\end{tabular}\endgroup%
{$\left.\llap{\phantom{%
\begingroup \smaller\smaller\smaller\begin{tabular}{@{}c@{}}%
\phantom{0}\\\phantom{0}\\\phantom{0}
\end{tabular}\endgroup%
}}\!\right]$}%
{$\left[\!\llap{\phantom{%
\begingroup \smaller\smaller\smaller\begin{tabular}{@{}c@{}}%
0\\0\\0
\end{tabular}\endgroup%
}}\right.$}%
\begingroup \smaller\smaller\smaller\begin{tabular}{@{}c@{}}%
1\\-3\\0
\end{tabular}\endgroup%
\kern3pt%
\begingroup \smaller\smaller\smaller\begin{tabular}{@{}c@{}}%
2\\3\\1
\end{tabular}\endgroup%
\kern3pt%
\begingroup \smaller\smaller\smaller\begin{tabular}{@{}c@{}}%
4\\5\\-1
\end{tabular}\endgroup%
{$\left.\llap{\phantom{%
\begingroup \smaller\smaller\smaller\begin{tabular}{@{}c@{}}%
0\\0\\0
\end{tabular}\endgroup%
}}\!\right]$}%
}%
\ifdim\wd\matricesbox>\halfwidth\myboxwidth=\hsize\else\myboxwidth=\halfwidth\fi
\vbox{%
\ifdim\myboxwidth=\hsize
\setbox\onelinebox=\hbox{%
\vbox{\hbox{%
$\Pi_{3,26}$ spans $L_{1.6}$%
}\hbox{%
$42\infty$ (shared)%
}%
}%
\hfill\copy\matricesbox
}%
\ifdim\wd\onelinebox>\myboxwidth
\hbox to \myboxwidth{%
$\Pi_{3,26}$ spans $L_{1.6}$%
\hfil
$42\infty$ (shared)%
}%
\box\matricesbox
\else
\hbox to \myboxwidth{%
\unhbox\onelinebox
}%
\fi
\else
\hbox to \myboxwidth{%
$\Pi_{3,26}$ spans $L_{1.6}$%
\hfil}%
\hbox to \myboxwidth{%
$42\infty$ (shared)%
\hfil}%
\box\matricesbox
\fi
}%
\hfill\discretionary{}{}{}%
\setbox\matricesbox=\hbox{%
{$\left[\!\llap{\phantom{%
\begingroup \smaller\smaller\smaller\begin{tabular}{@{}c@{}}%
\phantom{0}\\\phantom{0}\\\phantom{0}
\end{tabular}\endgroup%
}}\right.$}%
\begingroup \smaller\smaller\smaller\begin{tabular}{@{}c@{}}%
-1/15\\\phantom{0}\\\phantom{0}
\end{tabular}\endgroup%
\kern3pt%
\begingroup \smaller\smaller\smaller\begin{tabular}{@{}c@{}}%
\phantom{0}\\4/15\\-1/15
\end{tabular}\endgroup%
\kern3pt%
\begingroup \smaller\smaller\smaller\begin{tabular}{@{}c@{}}%
\phantom{0}\\-1/15\\4/15
\end{tabular}\endgroup%
{$\left.\llap{\phantom{%
\begingroup \smaller\smaller\smaller\begin{tabular}{@{}c@{}}%
\phantom{0}\\\phantom{0}\\\phantom{0}
\end{tabular}\endgroup%
}}\!\right]$}%
{$\left[\!\llap{\phantom{%
\begingroup \smaller\smaller\smaller\begin{tabular}{@{}c@{}}%
0\\0\\0
\end{tabular}\endgroup%
}}\right.$}%
\begingroup \smaller\smaller\smaller\begin{tabular}{@{}c@{}}%
1\\-2\\-1
\end{tabular}\endgroup%
\kern3pt%
\begingroup \smaller\smaller\smaller\begin{tabular}{@{}c@{}}%
2\\1\\3
\end{tabular}\endgroup%
\kern3pt%
\begingroup \smaller\smaller\smaller\begin{tabular}{@{}c@{}}%
1\\2\\0
\end{tabular}\endgroup%
{$\left.\llap{\phantom{%
\begingroup \smaller\smaller\smaller\begin{tabular}{@{}c@{}}%
0\\0\\0
\end{tabular}\endgroup%
}}\!\right]$}%
}%
\ifdim\wd\matricesbox>\halfwidth\myboxwidth=\hsize\else\myboxwidth=\halfwidth\fi
\vbox{%
\ifdim\myboxwidth=\hsize
\setbox\onelinebox=\hbox{%
\vbox{\hbox{%
$\Pi_{3,27}=A_{2,0,\onebar}$ spans $L_{1.4}$%
}\hbox{%
$42\infty$ (shared)%
}%
}%
\hfill\copy\matricesbox
}%
\ifdim\wd\onelinebox>\myboxwidth
\hbox to \myboxwidth{%
$\Pi_{3,27}=A_{2,0,\onebar}$ spans $L_{1.4}$%
\hfil
$42\infty$ (shared)%
}%
\box\matricesbox
\else
\hbox to \myboxwidth{%
\unhbox\onelinebox
}%
\fi
\else
\hbox to \myboxwidth{%
$\Pi_{3,27}=A_{2,0,\onebar}$ spans $L_{1.4}$%
\hfil}%
\hbox to \myboxwidth{%
$42\infty$ (shared)%
\hfil}%
\box\matricesbox
\fi
}%
\hfill\discretionary{}{}{}%
\setbox\matricesbox=\hbox{%
{$\left[\!\llap{\phantom{%
\begingroup \smaller\smaller\smaller\begin{tabular}{@{}c@{}}%
\phantom{0}\\\phantom{0}\\\phantom{0}
\end{tabular}\endgroup%
}}\right.$}%
\begingroup \smaller\smaller\smaller\begin{tabular}{@{}c@{}}%
-1/136\\\phantom{0}\\\phantom{0}
\end{tabular}\endgroup%
\kern3pt%
\begingroup \smaller\smaller\smaller\begin{tabular}{@{}c@{}}%
\phantom{0}\\21/34\\-3/17
\end{tabular}\endgroup%
\kern3pt%
\begingroup \smaller\smaller\smaller\begin{tabular}{@{}c@{}}%
\phantom{0}\\-3/17\\30/17
\end{tabular}\endgroup%
{$\left.\llap{\phantom{%
\begingroup \smaller\smaller\smaller\begin{tabular}{@{}c@{}}%
\phantom{0}\\\phantom{0}\\\phantom{0}
\end{tabular}\endgroup%
}}\!\right]$}%
{$\left[\!\llap{\phantom{%
\begingroup \smaller\smaller\smaller\begin{tabular}{@{}c@{}}%
0\\0\\0
\end{tabular}\endgroup%
}}\right.$}%
\begingroup \smaller\smaller\smaller\begin{tabular}{@{}c@{}}%
6\\-3\\-1
\end{tabular}\endgroup%
\kern3pt%
\begingroup \smaller\smaller\smaller\begin{tabular}{@{}c@{}}%
12\\4\\-1
\end{tabular}\endgroup%
\kern3pt%
\begingroup \smaller\smaller\smaller\begin{tabular}{@{}c@{}}%
2\\1\\1
\end{tabular}\endgroup%
{$\left.\llap{\phantom{%
\begingroup \smaller\smaller\smaller\begin{tabular}{@{}c@{}}%
0\\0\\0
\end{tabular}\endgroup%
}}\!\right]$}%
}%
\ifdim\wd\matricesbox>\halfwidth\myboxwidth=\hsize\else\myboxwidth=\halfwidth\fi
\vbox{%
\ifdim\myboxwidth=\hsize
\setbox\onelinebox=\hbox{%
\vbox{\hbox{%
$\Pi_{3,28}$ spans $L_{3.4}$%
}\hbox{%
$426$ (shared)%
}%
}%
\hfill\copy\matricesbox
}%
\ifdim\wd\onelinebox>\myboxwidth
\hbox to \myboxwidth{%
$\Pi_{3,28}$ spans $L_{3.4}$%
\hfil
$426$ (shared)%
}%
\box\matricesbox
\else
\hbox to \myboxwidth{%
\unhbox\onelinebox
}%
\fi
\else
\hbox to \myboxwidth{%
$\Pi_{3,28}$ spans $L_{3.4}$%
\hfil}%
\hbox to \myboxwidth{%
$426$ (shared)%
\hfil}%
\box\matricesbox
\fi
}%
\hfill\discretionary{}{}{}%
\setbox\matricesbox=\hbox{%
{$\left[\!\llap{\phantom{%
\begingroup \smaller\smaller\smaller\begin{tabular}{@{}c@{}}%
\phantom{0}\\\phantom{0}\\\phantom{0}
\end{tabular}\endgroup%
}}\right.$}%
\begingroup \smaller\smaller\smaller\begin{tabular}{@{}c@{}}%
-1/41\\\phantom{0}\\\phantom{0}
\end{tabular}\endgroup%
\kern3pt%
\begingroup \smaller\smaller\smaller\begin{tabular}{@{}c@{}}%
\phantom{0}\\10/41\\-4/41
\end{tabular}\endgroup%
\kern3pt%
\begingroup \smaller\smaller\smaller\begin{tabular}{@{}c@{}}%
\phantom{0}\\-4/41\\18/41
\end{tabular}\endgroup%
{$\left.\llap{\phantom{%
\begingroup \smaller\smaller\smaller\begin{tabular}{@{}c@{}}%
\phantom{0}\\\phantom{0}\\\phantom{0}
\end{tabular}\endgroup%
}}\!\right]$}%
{$\left[\!\llap{\phantom{%
\begingroup \smaller\smaller\smaller\begin{tabular}{@{}c@{}}%
0\\0\\0
\end{tabular}\endgroup%
}}\right.$}%
\begingroup \smaller\smaller\smaller\begin{tabular}{@{}c@{}}%
4\\-3\\-3
\end{tabular}\endgroup%
\kern3pt%
\begingroup \smaller\smaller\smaller\begin{tabular}{@{}c@{}}%
8\\-2\\4
\end{tabular}\endgroup%
\kern3pt%
\begingroup \smaller\smaller\smaller\begin{tabular}{@{}c@{}}%
1\\2\\1
\end{tabular}\endgroup%
{$\left.\llap{\phantom{%
\begingroup \smaller\smaller\smaller\begin{tabular}{@{}c@{}}%
0\\0\\0
\end{tabular}\endgroup%
}}\!\right]$}%
}%
\ifdim\wd\matricesbox>\halfwidth\myboxwidth=\hsize\else\myboxwidth=\halfwidth\fi
\vbox{%
\ifdim\myboxwidth=\hsize
\setbox\onelinebox=\hbox{%
\vbox{\hbox{%
$\Pi_{3,29}=A_{2,I,\onebar}$ spans $L_{1.9}$%
}\hbox{%
$42\infty$ (shared)%
}%
}%
\hfill\copy\matricesbox
}%
\ifdim\wd\onelinebox>\myboxwidth
\hbox to \myboxwidth{%
$\Pi_{3,29}=A_{2,I,\onebar}$ spans $L_{1.9}$%
\hfil
$42\infty$ (shared)%
}%
\box\matricesbox
\else
\hbox to \myboxwidth{%
\unhbox\onelinebox
}%
\fi
\else
\hbox to \myboxwidth{%
$\Pi_{3,29}=A_{2,I,\onebar}$ spans $L_{1.9}$%
\hfil}%
\hbox to \myboxwidth{%
$42\infty$ (shared)%
\hfil}%
\box\matricesbox
\fi
}%
\hfill\discretionary{}{}{}%
\setbox\matricesbox=\hbox{%
{$\left[\!\llap{\phantom{%
\begingroup \smaller\smaller\smaller\begin{tabular}{@{}c@{}}%
\phantom{0}\\\phantom{0}\\\phantom{0}
\end{tabular}\endgroup%
}}\right.$}%
\begingroup \smaller\smaller\smaller\begin{tabular}{@{}c@{}}%
-1/40\\\phantom{0}\\\phantom{0}
\end{tabular}\endgroup%
\kern3pt%
\begingroup \smaller\smaller\smaller\begin{tabular}{@{}c@{}}%
\phantom{0}\\21/10\\-3/5
\end{tabular}\endgroup%
\kern3pt%
\begingroup \smaller\smaller\smaller\begin{tabular}{@{}c@{}}%
\phantom{0}\\-3/5\\18/5
\end{tabular}\endgroup%
{$\left.\llap{\phantom{%
\begingroup \smaller\smaller\smaller\begin{tabular}{@{}c@{}}%
\phantom{0}\\\phantom{0}\\\phantom{0}
\end{tabular}\endgroup%
}}\!\right]$}%
{$\left[\!\llap{\phantom{%
\begingroup \smaller\smaller\smaller\begin{tabular}{@{}c@{}}%
0\\0\\0
\end{tabular}\endgroup%
}}\right.$}%
\begingroup \smaller\smaller\smaller\begin{tabular}{@{}c@{}}%
2\\-1\\0
\end{tabular}\endgroup%
\kern3pt%
\begingroup \smaller\smaller\smaller\begin{tabular}{@{}c@{}}%
6\\1\\-1
\end{tabular}\endgroup%
\kern3pt%
\begingroup \smaller\smaller\smaller\begin{tabular}{@{}c@{}}%
8\\2\\1
\end{tabular}\endgroup%
{$\left.\llap{\phantom{%
\begingroup \smaller\smaller\smaller\begin{tabular}{@{}c@{}}%
0\\0\\0
\end{tabular}\endgroup%
}}\!\right]$}%
}%
\ifdim\wd\matricesbox>\halfwidth\myboxwidth=\hsize\else\myboxwidth=\halfwidth\fi
\vbox{%
\ifdim\myboxwidth=\hsize
\setbox\onelinebox=\hbox{%
\vbox{\hbox{%
$\Pi_{3,30}$ spans $L_{7.9}$%
}\hbox{%
$62\infty$ (shared)%
}%
}%
\hfill\copy\matricesbox
}%
\ifdim\wd\onelinebox>\myboxwidth
\hbox to \myboxwidth{%
$\Pi_{3,30}$ spans $L_{7.9}$%
\hfil
$62\infty$ (shared)%
}%
\box\matricesbox
\else
\hbox to \myboxwidth{%
\unhbox\onelinebox
}%
\fi
\else
\hbox to \myboxwidth{%
$\Pi_{3,30}$ spans $L_{7.9}$%
\hfil}%
\hbox to \myboxwidth{%
$62\infty$ (shared)%
\hfil}%
\box\matricesbox
\fi
}%
\hfill\discretionary{}{}{}%
\setbox\matricesbox=\hbox{%
{$\left[\!\llap{\phantom{%
\begingroup \smaller\smaller\smaller\begin{tabular}{@{}c@{}}%
\phantom{0}\\\phantom{0}\\\phantom{0}
\end{tabular}\endgroup%
}}\right.$}%
\begingroup \smaller\smaller\smaller\begin{tabular}{@{}c@{}}%
-1/65\\\phantom{0}\\\phantom{0}
\end{tabular}\endgroup%
\kern3pt%
\begingroup \smaller\smaller\smaller\begin{tabular}{@{}c@{}}%
\phantom{0}\\14/65\\-1/65
\end{tabular}\endgroup%
\kern3pt%
\begingroup \smaller\smaller\smaller\begin{tabular}{@{}c@{}}%
\phantom{0}\\-1/65\\14/65
\end{tabular}\endgroup%
{$\left.\llap{\phantom{%
\begingroup \smaller\smaller\smaller\begin{tabular}{@{}c@{}}%
\phantom{0}\\\phantom{0}\\\phantom{0}
\end{tabular}\endgroup%
}}\!\right]$}%
{$\left[\!\llap{\phantom{%
\begingroup \smaller\smaller\smaller\begin{tabular}{@{}c@{}}%
0\\0\\0
\end{tabular}\endgroup%
}}\right.$}%
\begingroup \smaller\smaller\smaller\begin{tabular}{@{}c@{}}%
2\\-3\\-1
\end{tabular}\endgroup%
\kern3pt%
\begingroup \smaller\smaller\smaller\begin{tabular}{@{}c@{}}%
1\\1\\2
\end{tabular}\endgroup%
\kern3pt%
\begingroup \smaller\smaller\smaller\begin{tabular}{@{}c@{}}%
6\\5\\-2
\end{tabular}\endgroup%
{$\left.\llap{\phantom{%
\begingroup \smaller\smaller\smaller\begin{tabular}{@{}c@{}}%
0\\0\\0
\end{tabular}\endgroup%
}}\!\right]$}%
}%
\ifdim\wd\matricesbox>\halfwidth\myboxwidth=\hsize\else\myboxwidth=\halfwidth\fi
\vbox{%
\ifdim\myboxwidth=\hsize
\setbox\onelinebox=\hbox{%
\vbox{\hbox{%
$\Pi_{3,31}$ spans $L_{3.2}$%
}\hbox{%
$426$ (shared)%
}%
}%
\hfill\copy\matricesbox
}%
\ifdim\wd\onelinebox>\myboxwidth
\hbox to \myboxwidth{%
$\Pi_{3,31}$ spans $L_{3.2}$%
\hfil
$426$ (shared)%
}%
\box\matricesbox
\else
\hbox to \myboxwidth{%
\unhbox\onelinebox
}%
\fi
\else
\hbox to \myboxwidth{%
$\Pi_{3,31}$ spans $L_{3.2}$%
\hfil}%
\hbox to \myboxwidth{%
$426$ (shared)%
\hfil}%
\box\matricesbox
\fi
}%
\hfill\discretionary{}{}{}%
\setbox\matricesbox=\hbox{%
{$\left[\!\llap{\phantom{%
\begingroup \smaller\smaller\smaller\begin{tabular}{@{}c@{}}%
\phantom{0}\\\phantom{0}\\\phantom{0}
\end{tabular}\endgroup%
}}\right.$}%
\begingroup \smaller\smaller\smaller\begin{tabular}{@{}c@{}}%
-3/74\\\phantom{0}\\\phantom{0}
\end{tabular}\endgroup%
\kern3pt%
\begingroup \smaller\smaller\smaller\begin{tabular}{@{}c@{}}%
\phantom{0}\\8/37\\-1/37
\end{tabular}\endgroup%
\kern3pt%
\begingroup \smaller\smaller\smaller\begin{tabular}{@{}c@{}}%
\phantom{0}\\-1/37\\14/37
\end{tabular}\endgroup%
{$\left.\llap{\phantom{%
\begingroup \smaller\smaller\smaller\begin{tabular}{@{}c@{}}%
\phantom{0}\\\phantom{0}\\\phantom{0}
\end{tabular}\endgroup%
}}\!\right]$}%
{$\left[\!\llap{\phantom{%
\begingroup \smaller\smaller\smaller\begin{tabular}{@{}c@{}}%
0\\0\\0
\end{tabular}\endgroup%
}}\right.$}%
\begingroup \smaller\smaller\smaller\begin{tabular}{@{}c@{}}%
2\\-3\\-1
\end{tabular}\endgroup%
\kern3pt%
\begingroup \smaller\smaller\smaller\begin{tabular}{@{}c@{}}%
6\\5\\-2
\end{tabular}\endgroup%
\kern3pt%
\begingroup \smaller\smaller\smaller\begin{tabular}{@{}c@{}}%
2\\2\\2
\end{tabular}\endgroup%
{$\left.\llap{\phantom{%
\begingroup \smaller\smaller\smaller\begin{tabular}{@{}c@{}}%
0\\0\\0
\end{tabular}\endgroup%
}}\!\right]$}%
}%
\ifdim\wd\matricesbox>\halfwidth\myboxwidth=\hsize\else\myboxwidth=\halfwidth\fi
\vbox{%
\ifdim\myboxwidth=\hsize
\setbox\onelinebox=\hbox{%
\vbox{\hbox{%
$\Pi_{3,32}=A_{3,0,\onebar}$ spans $L_{7.6}$%
}\hbox{%
$62\infty$ (shared)%
}%
}%
\hfill\copy\matricesbox
}%
\ifdim\wd\onelinebox>\myboxwidth
\hbox to \myboxwidth{%
$\Pi_{3,32}=A_{3,0,\onebar}$ spans $L_{7.6}$%
\hfil
$62\infty$ (shared)%
}%
\box\matricesbox
\else
\hbox to \myboxwidth{%
\unhbox\onelinebox
}%
\fi
\else
\hbox to \myboxwidth{%
$\Pi_{3,32}=A_{3,0,\onebar}$ spans $L_{7.6}$%
\hfil}%
\hbox to \myboxwidth{%
$62\infty$ (shared)%
\hfil}%
\box\matricesbox
\fi
}%
\hfill\discretionary{}{}{}%
\setbox\matricesbox=\hbox{%
{$\left[\!\llap{\phantom{%
\begingroup \smaller\smaller\smaller\begin{tabular}{@{}c@{}}%
\phantom{0}\\\phantom{0}\\\phantom{0}
\end{tabular}\endgroup%
}}\right.$}%
\begingroup \smaller\smaller\smaller\begin{tabular}{@{}c@{}}%
-1/88\\\phantom{0}\\\phantom{0}
\end{tabular}\endgroup%
\kern3pt%
\begingroup \smaller\smaller\smaller\begin{tabular}{@{}c@{}}%
\phantom{0}\\45/22\\-9/11
\end{tabular}\endgroup%
\kern3pt%
\begingroup \smaller\smaller\smaller\begin{tabular}{@{}c@{}}%
\phantom{0}\\-9/11\\30/11
\end{tabular}\endgroup%
{$\left.\llap{\phantom{%
\begingroup \smaller\smaller\smaller\begin{tabular}{@{}c@{}}%
\phantom{0}\\\phantom{0}\\\phantom{0}
\end{tabular}\endgroup%
}}\!\right]$}%
{$\left[\!\llap{\phantom{%
\begingroup \smaller\smaller\smaller\begin{tabular}{@{}c@{}}%
0\\0\\0
\end{tabular}\endgroup%
}}\right.$}%
\begingroup \smaller\smaller\smaller\begin{tabular}{@{}c@{}}%
6\\-1\\1
\end{tabular}\endgroup%
\kern3pt%
\begingroup \smaller\smaller\smaller\begin{tabular}{@{}c@{}}%
18\\-1\\-3
\end{tabular}\endgroup%
\kern3pt%
\begingroup \smaller\smaller\smaller\begin{tabular}{@{}c@{}}%
2\\1\\0
\end{tabular}\endgroup%
{$\left.\llap{\phantom{%
\begingroup \smaller\smaller\smaller\begin{tabular}{@{}c@{}}%
0\\0\\0
\end{tabular}\endgroup%
}}\!\right]$}%
}%
\ifdim\wd\matricesbox>\halfwidth\myboxwidth=\hsize\else\myboxwidth=\halfwidth\fi
\vbox{%
\ifdim\myboxwidth=\hsize
\setbox\onelinebox=\hbox{%
\vbox{\hbox{%
$\Pi_{3,33}$ spans $L_{155.1}$%
}\hbox{%
$626$ (shared)%
}%
}%
\hfill\copy\matricesbox
}%
\ifdim\wd\onelinebox>\myboxwidth
\hbox to \myboxwidth{%
$\Pi_{3,33}$ spans $L_{155.1}$%
\hfil
$626$ (shared)%
}%
\box\matricesbox
\else
\hbox to \myboxwidth{%
\unhbox\onelinebox
}%
\fi
\else
\hbox to \myboxwidth{%
$\Pi_{3,33}$ spans $L_{155.1}$%
\hfil}%
\hbox to \myboxwidth{%
$626$ (shared)%
\hfil}%
\box\matricesbox
\fi
}%
\hfill\discretionary{}{}{}%
\setbox\matricesbox=\hbox{%
{$\left[\!\llap{\phantom{%
\begingroup \smaller\smaller\smaller\begin{tabular}{@{}c@{}}%
\phantom{0}\\\phantom{0}\\\phantom{0}
\end{tabular}\endgroup%
}}\right.$}%
\begingroup \smaller\smaller\smaller\begin{tabular}{@{}c@{}}%
-1/35\\\phantom{0}\\\phantom{0}
\end{tabular}\endgroup%
\kern3pt%
\begingroup \smaller\smaller\smaller\begin{tabular}{@{}c@{}}%
\phantom{0}\\36/35\\-6/35
\end{tabular}\endgroup%
\kern3pt%
\begingroup \smaller\smaller\smaller\begin{tabular}{@{}c@{}}%
\phantom{0}\\-6/35\\36/35
\end{tabular}\endgroup%
{$\left.\llap{\phantom{%
\begingroup \smaller\smaller\smaller\begin{tabular}{@{}c@{}}%
\phantom{0}\\\phantom{0}\\\phantom{0}
\end{tabular}\endgroup%
}}\!\right]$}%
{$\left[\!\llap{\phantom{%
\begingroup \smaller\smaller\smaller\begin{tabular}{@{}c@{}}%
0\\0\\0
\end{tabular}\endgroup%
}}\right.$}%
\begingroup \smaller\smaller\smaller\begin{tabular}{@{}c@{}}%
4\\2\\1
\end{tabular}\endgroup%
\kern3pt%
\begingroup \smaller\smaller\smaller\begin{tabular}{@{}c@{}}%
12\\-1\\-4
\end{tabular}\endgroup%
\kern3pt%
\begingroup \smaller\smaller\smaller\begin{tabular}{@{}c@{}}%
1\\-1\\0
\end{tabular}\endgroup%
{$\left.\llap{\phantom{%
\begingroup \smaller\smaller\smaller\begin{tabular}{@{}c@{}}%
0\\0\\0
\end{tabular}\endgroup%
}}\!\right]$}%
}%
\ifdim\wd\matricesbox>\halfwidth\myboxwidth=\hsize\else\myboxwidth=\halfwidth\fi
\vbox{%
\ifdim\myboxwidth=\hsize
\setbox\onelinebox=\hbox{%
\vbox{\hbox{%
$\Pi_{3,34}=A_{3,I,\onebar}$ spans $L_{7.7}$%
}\hbox{%
$62\infty$ (shared)%
}%
}%
\hfill\copy\matricesbox
}%
\ifdim\wd\onelinebox>\myboxwidth
\hbox to \myboxwidth{%
$\Pi_{3,34}=A_{3,I,\onebar}$ spans $L_{7.7}$%
\hfil
$62\infty$ (shared)%
}%
\box\matricesbox
\else
\hbox to \myboxwidth{%
\unhbox\onelinebox
}%
\fi
\else
\hbox to \myboxwidth{%
$\Pi_{3,34}=A_{3,I,\onebar}$ spans $L_{7.7}$%
\hfil}%
\hbox to \myboxwidth{%
$62\infty$ (shared)%
\hfil}%
\box\matricesbox
\fi
}%
\hfill\discretionary{}{}{}%
\setbox\matricesbox=\hbox{%
{$\left[\!\llap{\phantom{%
\begingroup \smaller\smaller\smaller\begin{tabular}{@{}c@{}}%
\phantom{0}\\\phantom{0}\\\phantom{0}
\end{tabular}\endgroup%
}}\right.$}%
\begingroup \smaller\smaller\smaller\begin{tabular}{@{}c@{}}%
-1/40\\\phantom{0}\\\phantom{0}
\end{tabular}\endgroup%
\kern3pt%
\begingroup \smaller\smaller\smaller\begin{tabular}{@{}c@{}}%
\phantom{0}\\9/40\\-1/40
\end{tabular}\endgroup%
\kern3pt%
\begingroup \smaller\smaller\smaller\begin{tabular}{@{}c@{}}%
\phantom{0}\\-1/40\\9/40
\end{tabular}\endgroup%
{$\left.\llap{\phantom{%
\begingroup \smaller\smaller\smaller\begin{tabular}{@{}c@{}}%
\phantom{0}\\\phantom{0}\\\phantom{0}
\end{tabular}\endgroup%
}}\!\right]$}%
{$\left[\!\llap{\phantom{%
\begingroup \smaller\smaller\smaller\begin{tabular}{@{}c@{}}%
0\\0\\0
\end{tabular}\endgroup%
}}\right.$}%
\begingroup \smaller\smaller\smaller\begin{tabular}{@{}c@{}}%
2\\3\\1
\end{tabular}\endgroup%
\kern3pt%
\begingroup \smaller\smaller\smaller\begin{tabular}{@{}c@{}}%
1\\-1\\-2
\end{tabular}\endgroup%
\kern3pt%
\begingroup \smaller\smaller\smaller\begin{tabular}{@{}c@{}}%
8\\-6\\2
\end{tabular}\endgroup%
{$\left.\llap{\phantom{%
\begingroup \smaller\smaller\smaller\begin{tabular}{@{}c@{}}%
0\\0\\0
\end{tabular}\endgroup%
}}\!\right]$}%
}%
\ifdim\wd\matricesbox>\halfwidth\myboxwidth=\hsize\else\myboxwidth=\halfwidth\fi
\vbox{%
\ifdim\myboxwidth=\hsize
\setbox\onelinebox=\hbox{%
\vbox{\hbox{%
$\Pi_{3,35}$ spans $L_{1.3}$%
}\hbox{%
$42\infty$ (shared)%
}%
}%
\hfill\copy\matricesbox
}%
\ifdim\wd\onelinebox>\myboxwidth
\hbox to \myboxwidth{%
$\Pi_{3,35}$ spans $L_{1.3}$%
\hfil
$42\infty$ (shared)%
}%
\box\matricesbox
\else
\hbox to \myboxwidth{%
\unhbox\onelinebox
}%
\fi
\else
\hbox to \myboxwidth{%
$\Pi_{3,35}$ spans $L_{1.3}$%
\hfil}%
\hbox to \myboxwidth{%
$42\infty$ (shared)%
\hfil}%
\box\matricesbox
\fi
}%
\hfill\discretionary{}{}{}%
\setbox\matricesbox=\hbox{%
{$\left[\!\llap{\phantom{%
\begingroup \smaller\smaller\smaller\begin{tabular}{@{}c@{}}%
\phantom{0}\\\phantom{0}\\\phantom{0}
\end{tabular}\endgroup%
}}\right.$}%
\begingroup \smaller\smaller\smaller\begin{tabular}{@{}c@{}}%
-1/11\\\phantom{0}\\\phantom{0}
\end{tabular}\endgroup%
\kern3pt%
\begingroup \smaller\smaller\smaller\begin{tabular}{@{}c@{}}%
\phantom{0}\\4/11\\-2/11
\end{tabular}\endgroup%
\kern3pt%
\begingroup \smaller\smaller\smaller\begin{tabular}{@{}c@{}}%
\phantom{0}\\-2/11\\12/11
\end{tabular}\endgroup%
{$\left.\llap{\phantom{%
\begingroup \smaller\smaller\smaller\begin{tabular}{@{}c@{}}%
\phantom{0}\\\phantom{0}\\\phantom{0}
\end{tabular}\endgroup%
}}\!\right]$}%
{$\left[\!\llap{\phantom{%
\begingroup \smaller\smaller\smaller\begin{tabular}{@{}c@{}}%
0\\0\\0
\end{tabular}\endgroup%
}}\right.$}%
\begingroup \smaller\smaller\smaller\begin{tabular}{@{}c@{}}%
1\\-1\\-1
\end{tabular}\endgroup%
\kern3pt%
\begingroup \smaller\smaller\smaller\begin{tabular}{@{}c@{}}%
4\\4\\1
\end{tabular}\endgroup%
\kern3pt%
\begingroup \smaller\smaller\smaller\begin{tabular}{@{}c@{}}%
1\\0\\1
\end{tabular}\endgroup%
{$\left.\llap{\phantom{%
\begingroup \smaller\smaller\smaller\begin{tabular}{@{}c@{}}%
0\\0\\0
\end{tabular}\endgroup%
}}\!\right]$}%
}%
\ifdim\wd\matricesbox>\halfwidth\myboxwidth=\hsize\else\myboxwidth=\halfwidth\fi
\vbox{%
\ifdim\myboxwidth=\hsize
\setbox\onelinebox=\hbox{%
\vbox{\hbox{%
$\Pi_{3,36}=A_{4,0,\onebar}=\hbox{GN}_{10}$ spans $L_{2.3}$%
}\hbox{%
$\infty2\infty$ (shared)%
}%
}%
\hfill\copy\matricesbox
}%
\ifdim\wd\onelinebox>\myboxwidth
\hbox to \myboxwidth{%
$\Pi_{3,36}=A_{4,0,\onebar}=\hbox{GN}_{10}$ spans $L_{2.3}$%
\hfil
$\infty2\infty$ (shared)%
}%
\box\matricesbox
\else
\hbox to \myboxwidth{%
\unhbox\onelinebox
}%
\fi
\else
\hbox to \myboxwidth{%
$\Pi_{3,36}=A_{4,0,\onebar}=\hbox{GN}_{10}$ spans $L_{2.3}$%
\hfil}%
\hbox to \myboxwidth{%
$\infty2\infty$ (shared)%
\hfil}%
\box\matricesbox
\fi
}%
\hfill\discretionary{}{}{}%
\setbox\matricesbox=\hbox{%
{$\left[\!\llap{\phantom{%
\begingroup \smaller\smaller\smaller\begin{tabular}{@{}c@{}}%
\phantom{0}\\\phantom{0}\\\phantom{0}
\end{tabular}\endgroup%
}}\right.$}%
\begingroup \smaller\smaller\smaller\begin{tabular}{@{}c@{}}%
-1/72\\\phantom{0}\\\phantom{0}
\end{tabular}\endgroup%
\kern3pt%
\begingroup \smaller\smaller\smaller\begin{tabular}{@{}c@{}}%
\phantom{0}\\13/18\\-1/3
\end{tabular}\endgroup%
\kern3pt%
\begingroup \smaller\smaller\smaller\begin{tabular}{@{}c@{}}%
\phantom{0}\\-1/3\\2
\end{tabular}\endgroup%
{$\left.\llap{\phantom{%
\begingroup \smaller\smaller\smaller\begin{tabular}{@{}c@{}}%
\phantom{0}\\\phantom{0}\\\phantom{0}
\end{tabular}\endgroup%
}}\!\right]$}%
{$\left[\!\llap{\phantom{%
\begingroup \smaller\smaller\smaller\begin{tabular}{@{}c@{}}%
0\\0\\0
\end{tabular}\endgroup%
}}\right.$}%
\begingroup \smaller\smaller\smaller\begin{tabular}{@{}c@{}}%
6\\3\\1
\end{tabular}\endgroup%
\kern3pt%
\begingroup \smaller\smaller\smaller\begin{tabular}{@{}c@{}}%
2\\-1\\-1
\end{tabular}\endgroup%
\kern3pt%
\begingroup \smaller\smaller\smaller\begin{tabular}{@{}c@{}}%
24\\-6\\1
\end{tabular}\endgroup%
{$\left.\llap{\phantom{%
\begingroup \smaller\smaller\smaller\begin{tabular}{@{}c@{}}%
0\\0\\0
\end{tabular}\endgroup%
}}\!\right]$}%
}%
\ifdim\wd\matricesbox>\halfwidth\myboxwidth=\hsize\else\myboxwidth=\halfwidth\fi
\vbox{%
\ifdim\myboxwidth=\hsize
\setbox\onelinebox=\hbox{%
\vbox{\hbox{%
$\Pi_{3,37}$ spans $L_{7.13}$%
}\hbox{%
$62\infty$ (shared)%
}%
}%
\hfill\copy\matricesbox
}%
\ifdim\wd\onelinebox>\myboxwidth
\hbox to \myboxwidth{%
$\Pi_{3,37}$ spans $L_{7.13}$%
\hfil
$62\infty$ (shared)%
}%
\box\matricesbox
\else
\hbox to \myboxwidth{%
\unhbox\onelinebox
}%
\fi
\else
\hbox to \myboxwidth{%
$\Pi_{3,37}$ spans $L_{7.13}$%
\hfil}%
\hbox to \myboxwidth{%
$62\infty$ (shared)%
\hfil}%
\box\matricesbox
\fi
}%
\hfill\discretionary{}{}{}%
\setbox\matricesbox=\hbox{%
{$\left[\!\llap{\phantom{%
\begingroup \smaller\smaller\smaller\begin{tabular}{@{}c@{}}%
\phantom{0}\\\phantom{0}\\\phantom{0}
\end{tabular}\endgroup%
}}\right.$}%
\begingroup \smaller\smaller\smaller\begin{tabular}{@{}c@{}}%
-1/32\\\phantom{0}\\\phantom{0}
\end{tabular}\endgroup%
\kern3pt%
\begingroup \smaller\smaller\smaller\begin{tabular}{@{}c@{}}%
\phantom{0}\\9/32\\-1/8
\end{tabular}\endgroup%
\kern3pt%
\begingroup \smaller\smaller\smaller\begin{tabular}{@{}c@{}}%
\phantom{0}\\-1/8\\1/2
\end{tabular}\endgroup%
{$\left.\llap{\phantom{%
\begingroup \smaller\smaller\smaller\begin{tabular}{@{}c@{}}%
\phantom{0}\\\phantom{0}\\\phantom{0}
\end{tabular}\endgroup%
}}\!\right]$}%
{$\left[\!\llap{\phantom{%
\begingroup \smaller\smaller\smaller\begin{tabular}{@{}c@{}}%
0\\0\\0
\end{tabular}\endgroup%
}}\right.$}%
\begingroup \smaller\smaller\smaller\begin{tabular}{@{}c@{}}%
4\\0\\3
\end{tabular}\endgroup%
\kern3pt%
\begingroup \smaller\smaller\smaller\begin{tabular}{@{}c@{}}%
16\\-8\\-6
\end{tabular}\endgroup%
\kern3pt%
\begingroup \smaller\smaller\smaller\begin{tabular}{@{}c@{}}%
1\\1\\-1
\end{tabular}\endgroup%
{$\left.\llap{\phantom{%
\begingroup \smaller\smaller\smaller\begin{tabular}{@{}c@{}}%
0\\0\\0
\end{tabular}\endgroup%
}}\!\right]$}%
}%
\ifdim\wd\matricesbox>\halfwidth\myboxwidth=\hsize\else\myboxwidth=\halfwidth\fi
\vbox{%
\ifdim\myboxwidth=\hsize
\setbox\onelinebox=\hbox{%
\vbox{\hbox{%
$\Pi_{3,38}=A_{4,I,\onebar}=\hbox{GN}_{3}$ spans $L_{140.4}$%
}\hbox{%
$\infty2\infty$ (shared)%
}%
}%
\hfill\copy\matricesbox
}%
\ifdim\wd\onelinebox>\myboxwidth
\hbox to \myboxwidth{%
$\Pi_{3,38}=A_{4,I,\onebar}=\hbox{GN}_{3}$ spans $L_{140.4}$%
\hfil
$\infty2\infty$ (shared)%
}%
\box\matricesbox
\else
\hbox to \myboxwidth{%
\unhbox\onelinebox
}%
\fi
\else
\hbox to \myboxwidth{%
$\Pi_{3,38}=A_{4,I,\onebar}=\hbox{GN}_{3}$ spans $L_{140.4}$%
\hfil}%
\hbox to \myboxwidth{%
$\infty2\infty$ (shared)%
\hfil}%
\box\matricesbox
\fi
}%
\hfill\discretionary{}{}{}%
\setbox\matricesbox=\hbox{%
{$\left[\!\llap{\phantom{%
\begingroup \smaller\smaller\smaller\begin{tabular}{@{}c@{}}%
\phantom{0}\\\phantom{0}\\\phantom{0}
\end{tabular}\endgroup%
}}\right.$}%
\begingroup \smaller\smaller\smaller\begin{tabular}{@{}c@{}}%
-1/22\\\phantom{0}\\\phantom{0}
\end{tabular}\endgroup%
\kern3pt%
\begingroup \smaller\smaller\smaller\begin{tabular}{@{}c@{}}%
\phantom{0}\\3/22\\-1/22
\end{tabular}\endgroup%
\kern3pt%
\begingroup \smaller\smaller\smaller\begin{tabular}{@{}c@{}}%
\phantom{0}\\-1/22\\15/22
\end{tabular}\endgroup%
{$\left.\llap{\phantom{%
\begingroup \smaller\smaller\smaller\begin{tabular}{@{}c@{}}%
\phantom{0}\\\phantom{0}\\\phantom{0}
\end{tabular}\endgroup%
}}\!\right]$}%
{$\left[\!\llap{\phantom{%
\begingroup \smaller\smaller\smaller\begin{tabular}{@{}c@{}}%
0\\0\\0
\end{tabular}\endgroup%
}}\right.$}%
\begingroup \smaller\smaller\smaller\begin{tabular}{@{}c@{}}%
2\\4\\0
\end{tabular}\endgroup%
\kern3pt%
\begingroup \smaller\smaller\smaller\begin{tabular}{@{}c@{}}%
1\\-2\\-1
\end{tabular}\endgroup%
\kern3pt%
\begingroup \smaller\smaller\smaller\begin{tabular}{@{}c@{}}%
2\\-3\\1
\end{tabular}\endgroup%
{$\left.\llap{\phantom{%
\begingroup \smaller\smaller\smaller\begin{tabular}{@{}c@{}}%
0\\0\\0
\end{tabular}\endgroup%
}}\!\right]$}%
}%
\ifdim\wd\matricesbox>\halfwidth\myboxwidth=\hsize\else\myboxwidth=\halfwidth\fi
\vbox{%
\ifdim\myboxwidth=\hsize
\setbox\onelinebox=\hbox{%
\vbox{\hbox{%
$\Pi_{3,39}$ spans $L_{1.5}$%
}\hbox{%
$42\infty$ (shared)%
}%
}%
\hfill\copy\matricesbox
}%
\ifdim\wd\onelinebox>\myboxwidth
\hbox to \myboxwidth{%
$\Pi_{3,39}$ spans $L_{1.5}$%
\hfil
$42\infty$ (shared)%
}%
\box\matricesbox
\else
\hbox to \myboxwidth{%
\unhbox\onelinebox
}%
\fi
\else
\hbox to \myboxwidth{%
$\Pi_{3,39}$ spans $L_{1.5}$%
\hfil}%
\hbox to \myboxwidth{%
$42\infty$ (shared)%
\hfil}%
\box\matricesbox
\fi
}%
\hfill\discretionary{}{}{}%
\setbox\matricesbox=\hbox{%
{$\left[\!\llap{\phantom{%
\begingroup \smaller\smaller\smaller\begin{tabular}{@{}c@{}}%
\phantom{0}\\\phantom{0}\\\phantom{0}
\end{tabular}\endgroup%
}}\right.$}%
\begingroup \smaller\smaller\smaller\begin{tabular}{@{}c@{}}%
-1/91\\\phantom{0}\\\phantom{0}
\end{tabular}\endgroup%
\kern3pt%
\begingroup \smaller\smaller\smaller\begin{tabular}{@{}c@{}}%
\phantom{0}\\30/91\\-9/91
\end{tabular}\endgroup%
\kern3pt%
\begingroup \smaller\smaller\smaller\begin{tabular}{@{}c@{}}%
\phantom{0}\\-9/91\\30/91
\end{tabular}\endgroup%
{$\left.\llap{\phantom{%
\begingroup \smaller\smaller\smaller\begin{tabular}{@{}c@{}}%
\phantom{0}\\\phantom{0}\\\phantom{0}
\end{tabular}\endgroup%
}}\!\right]$}%
{$\left[\!\llap{\phantom{%
\begingroup \smaller\smaller\smaller\begin{tabular}{@{}c@{}}%
0\\0\\0
\end{tabular}\endgroup%
}}\right.$}%
\begingroup \smaller\smaller\smaller\begin{tabular}{@{}c@{}}%
6\\-1\\4
\end{tabular}\endgroup%
\kern3pt%
\begingroup \smaller\smaller\smaller\begin{tabular}{@{}c@{}}%
3\\-2\\-3
\end{tabular}\endgroup%
\kern3pt%
\begingroup \smaller\smaller\smaller\begin{tabular}{@{}c@{}}%
2\\2\\-1
\end{tabular}\endgroup%
{$\left.\llap{\phantom{%
\begingroup \smaller\smaller\smaller\begin{tabular}{@{}c@{}}%
0\\0\\0
\end{tabular}\endgroup%
}}\!\right]$}%
}%
\ifdim\wd\matricesbox>\halfwidth\myboxwidth=\hsize\else\myboxwidth=\halfwidth\fi
\vbox{%
\ifdim\myboxwidth=\hsize
\setbox\onelinebox=\hbox{%
\vbox{\hbox{%
$\Pi_{3,40}$ spans $L_{3.3}$%
}\hbox{%
$426$ (shared)%
}%
}%
\hfill\copy\matricesbox
}%
\ifdim\wd\onelinebox>\myboxwidth
\hbox to \myboxwidth{%
$\Pi_{3,40}$ spans $L_{3.3}$%
\hfil
$426$ (shared)%
}%
\box\matricesbox
\else
\hbox to \myboxwidth{%
\unhbox\onelinebox
}%
\fi
\else
\hbox to \myboxwidth{%
$\Pi_{3,40}$ spans $L_{3.3}$%
\hfil}%
\hbox to \myboxwidth{%
$426$ (shared)%
\hfil}%
\box\matricesbox
\fi
}%
\hfill\discretionary{}{}{}%
\setbox\matricesbox=\hbox{%
{$\left[\!\llap{\phantom{%
\begingroup \smaller\smaller\smaller\begin{tabular}{@{}c@{}}%
\phantom{0}\\\phantom{0}\\\phantom{0}
\end{tabular}\endgroup%
}}\right.$}%
\begingroup \smaller\smaller\smaller\begin{tabular}{@{}c@{}}%
-1/29\\\phantom{0}\\\phantom{0}
\end{tabular}\endgroup%
\kern3pt%
\begingroup \smaller\smaller\smaller\begin{tabular}{@{}c@{}}%
\phantom{0}\\6/29\\-2/29
\end{tabular}\endgroup%
\kern3pt%
\begingroup \smaller\smaller\smaller\begin{tabular}{@{}c@{}}%
\phantom{0}\\-2/29\\20/29
\end{tabular}\endgroup%
{$\left.\llap{\phantom{%
\begingroup \smaller\smaller\smaller\begin{tabular}{@{}c@{}}%
\phantom{0}\\\phantom{0}\\\phantom{0}
\end{tabular}\endgroup%
}}\!\right]$}%
{$\left[\!\llap{\phantom{%
\begingroup \smaller\smaller\smaller\begin{tabular}{@{}c@{}}%
0\\0\\0
\end{tabular}\endgroup%
}}\right.$}%
\begingroup \smaller\smaller\smaller\begin{tabular}{@{}c@{}}%
4\\4\\-1
\end{tabular}\endgroup%
\kern3pt%
\begingroup \smaller\smaller\smaller\begin{tabular}{@{}c@{}}%
2\\-3\\-1
\end{tabular}\endgroup%
\kern3pt%
\begingroup \smaller\smaller\smaller\begin{tabular}{@{}c@{}}%
1\\-1\\1
\end{tabular}\endgroup%
{$\left.\llap{\phantom{%
\begingroup \smaller\smaller\smaller\begin{tabular}{@{}c@{}}%
0\\0\\0
\end{tabular}\endgroup%
}}\!\right]$}%
}%
\ifdim\wd\matricesbox>\halfwidth\myboxwidth=\hsize\else\myboxwidth=\halfwidth\fi
\vbox{%
\ifdim\myboxwidth=\hsize
\setbox\onelinebox=\hbox{%
\vbox{\hbox{%
$\Pi_{3,41}$ spans $L_{1.7}$%
}\hbox{%
$42\infty$ (shared)%
}%
}%
\hfill\copy\matricesbox
}%
\ifdim\wd\onelinebox>\myboxwidth
\hbox to \myboxwidth{%
$\Pi_{3,41}$ spans $L_{1.7}$%
\hfil
$42\infty$ (shared)%
}%
\box\matricesbox
\else
\hbox to \myboxwidth{%
\unhbox\onelinebox
}%
\fi
\else
\hbox to \myboxwidth{%
$\Pi_{3,41}$ spans $L_{1.7}$%
\hfil}%
\hbox to \myboxwidth{%
$42\infty$ (shared)%
\hfil}%
\box\matricesbox
\fi
}%
\hfill\discretionary{}{}{}%
\setbox\matricesbox=\hbox{%
{$\left[\!\llap{\phantom{%
\begingroup \smaller\smaller\smaller\begin{tabular}{@{}c@{}}%
\phantom{0}\\\phantom{0}\\\phantom{0}
\end{tabular}\endgroup%
}}\right.$}%
\begingroup \smaller\smaller\smaller\begin{tabular}{@{}c@{}}%
-1/42\\\phantom{0}\\\phantom{0}
\end{tabular}\endgroup%
\kern3pt%
\begingroup \smaller\smaller\smaller\begin{tabular}{@{}c@{}}%
\phantom{0}\\8/21\\-1/21
\end{tabular}\endgroup%
\kern3pt%
\begingroup \smaller\smaller\smaller\begin{tabular}{@{}c@{}}%
\phantom{0}\\-1/21\\8/21
\end{tabular}\endgroup%
{$\left.\llap{\phantom{%
\begingroup \smaller\smaller\smaller\begin{tabular}{@{}c@{}}%
\phantom{0}\\\phantom{0}\\\phantom{0}
\end{tabular}\endgroup%
}}\!\right]$}%
{$\left[\!\llap{\phantom{%
\begingroup \smaller\smaller\smaller\begin{tabular}{@{}c@{}}%
0\\0\\0
\end{tabular}\endgroup%
}}\right.$}%
\begingroup \smaller\smaller\smaller\begin{tabular}{@{}c@{}}%
6\\-4\\1
\end{tabular}\endgroup%
\kern3pt%
\begingroup \smaller\smaller\smaller\begin{tabular}{@{}c@{}}%
2\\1\\-2
\end{tabular}\endgroup%
\kern3pt%
\begingroup \smaller\smaller\smaller\begin{tabular}{@{}c@{}}%
6\\4\\2
\end{tabular}\endgroup%
{$\left.\llap{\phantom{%
\begingroup \smaller\smaller\smaller\begin{tabular}{@{}c@{}}%
0\\0\\0
\end{tabular}\endgroup%
}}\!\right]$}%
}%
\ifdim\wd\matricesbox>\halfwidth\myboxwidth=\hsize\else\myboxwidth=\halfwidth\fi
\vbox{%
\ifdim\myboxwidth=\hsize
\setbox\onelinebox=\hbox{%
\vbox{\hbox{%
$\Pi_{3,42}$ spans $L_{7.8}$%
}\hbox{%
$62\infty$ (shared)%
}%
}%
\hfill\copy\matricesbox
}%
\ifdim\wd\onelinebox>\myboxwidth
\hbox to \myboxwidth{%
$\Pi_{3,42}$ spans $L_{7.8}$%
\hfil
$62\infty$ (shared)%
}%
\box\matricesbox
\else
\hbox to \myboxwidth{%
\unhbox\onelinebox
}%
\fi
\else
\hbox to \myboxwidth{%
$\Pi_{3,42}$ spans $L_{7.8}$%
\hfil}%
\hbox to \myboxwidth{%
$62\infty$ (shared)%
\hfil}%
\box\matricesbox
\fi
}%
\hfill\discretionary{}{}{}%
\setbox\matricesbox=\hbox{%
{$\left[\!\llap{\phantom{%
\begingroup \smaller\smaller\smaller\begin{tabular}{@{}c@{}}%
\phantom{0}\\\phantom{0}\\\phantom{0}
\end{tabular}\endgroup%
}}\right.$}%
\begingroup \smaller\smaller\smaller\begin{tabular}{@{}c@{}}%
-1/20\\\phantom{0}\\\phantom{0}
\end{tabular}\endgroup%
\kern3pt%
\begingroup \smaller\smaller\smaller\begin{tabular}{@{}c@{}}%
\phantom{0}\\1/4\\\phantom{0}
\end{tabular}\endgroup%
\kern3pt%
\begingroup \smaller\smaller\smaller\begin{tabular}{@{}c@{}}%
\phantom{0}\\\phantom{0}\\4/5
\end{tabular}\endgroup%
{$\left.\llap{\phantom{%
\begingroup \smaller\smaller\smaller\begin{tabular}{@{}c@{}}%
\phantom{0}\\\phantom{0}\\\phantom{0}
\end{tabular}\endgroup%
}}\!\right]$}%
{$\left[\!\llap{\phantom{%
\begingroup \smaller\smaller\smaller\begin{tabular}{@{}c@{}}%
0\\0\\0
\end{tabular}\endgroup%
}}\right.$}%
\begingroup \smaller\smaller\smaller\begin{tabular}{@{}c@{}}%
4\\-4\\1
\end{tabular}\endgroup%
\kern3pt%
\begingroup \smaller\smaller\smaller\begin{tabular}{@{}c@{}}%
4\\4\\1
\end{tabular}\endgroup%
\kern3pt%
\begingroup \smaller\smaller\smaller\begin{tabular}{@{}c@{}}%
1\\1\\-1
\end{tabular}\endgroup%
{$\left.\llap{\phantom{%
\begingroup \smaller\smaller\smaller\begin{tabular}{@{}c@{}}%
0\\0\\0
\end{tabular}\endgroup%
}}\!\right]$}%
}%
\ifdim\wd\matricesbox>\halfwidth\myboxwidth=\hsize\else\myboxwidth=\halfwidth\fi
\vbox{%
\ifdim\myboxwidth=\hsize
\setbox\onelinebox=\hbox{%
\vbox{\hbox{%
$\Pi_{3,43}=\hbox{GN}_{5}$ spans $L_{140.3}$%
}\hbox{%
$\infty2\infty$ (shared)%
}%
}%
\hfill\copy\matricesbox
}%
\ifdim\wd\onelinebox>\myboxwidth
\hbox to \myboxwidth{%
$\Pi_{3,43}=\hbox{GN}_{5}$ spans $L_{140.3}$%
\hfil
$\infty2\infty$ (shared)%
}%
\box\matricesbox
\else
\hbox to \myboxwidth{%
\unhbox\onelinebox
}%
\fi
\else
\hbox to \myboxwidth{%
$\Pi_{3,43}=\hbox{GN}_{5}$ spans $L_{140.3}$%
\hfil}%
\hbox to \myboxwidth{%
$\infty2\infty$ (shared)%
\hfil}%
\box\matricesbox
\fi
}%
\hfill\discretionary{}{}{}%
\setbox\matricesbox=\hbox{%
{$\left[\!\llap{\phantom{%
\begingroup \smaller\smaller\smaller\begin{tabular}{@{}c@{}}%
\phantom{0}\\\phantom{0}\\\phantom{0}
\end{tabular}\endgroup%
}}\right.$}%
\begingroup \smaller\smaller\smaller\begin{tabular}{@{}c@{}}%
-1/57\\\phantom{0}\\\phantom{0}
\end{tabular}\endgroup%
\kern3pt%
\begingroup \smaller\smaller\smaller\begin{tabular}{@{}c@{}}%
\phantom{0}\\28/57\\-10/57
\end{tabular}\endgroup%
\kern3pt%
\begingroup \smaller\smaller\smaller\begin{tabular}{@{}c@{}}%
\phantom{0}\\-10/57\\28/57
\end{tabular}\endgroup%
{$\left.\llap{\phantom{%
\begingroup \smaller\smaller\smaller\begin{tabular}{@{}c@{}}%
\phantom{0}\\\phantom{0}\\\phantom{0}
\end{tabular}\endgroup%
}}\!\right]$}%
{$\left[\!\llap{\phantom{%
\begingroup \smaller\smaller\smaller\begin{tabular}{@{}c@{}}%
0\\0\\0
\end{tabular}\endgroup%
}}\right.$}%
\begingroup \smaller\smaller\smaller\begin{tabular}{@{}c@{}}%
12\\5\\-1
\end{tabular}\endgroup%
\kern3pt%
\begingroup \smaller\smaller\smaller\begin{tabular}{@{}c@{}}%
4\\-3\\-2
\end{tabular}\endgroup%
\kern3pt%
\begingroup \smaller\smaller\smaller\begin{tabular}{@{}c@{}}%
3\\-1\\2
\end{tabular}\endgroup%
{$\left.\llap{\phantom{%
\begingroup \smaller\smaller\smaller\begin{tabular}{@{}c@{}}%
0\\0\\0
\end{tabular}\endgroup%
}}\!\right]$}%
}%
\ifdim\wd\matricesbox>\halfwidth\myboxwidth=\hsize\else\myboxwidth=\halfwidth\fi
\vbox{%
\ifdim\myboxwidth=\hsize
\setbox\onelinebox=\hbox{%
\vbox{\hbox{%
$\Pi_{3,44}$ spans $L_{7.11}$%
}\hbox{%
$62\infty$ (shared)%
}%
}%
\hfill\copy\matricesbox
}%
\ifdim\wd\onelinebox>\myboxwidth
\hbox to \myboxwidth{%
$\Pi_{3,44}$ spans $L_{7.11}$%
\hfil
$62\infty$ (shared)%
}%
\box\matricesbox
\else
\hbox to \myboxwidth{%
\unhbox\onelinebox
}%
\fi
\else
\hbox to \myboxwidth{%
$\Pi_{3,44}$ spans $L_{7.11}$%
\hfil}%
\hbox to \myboxwidth{%
$62\infty$ (shared)%
\hfil}%
\box\matricesbox
\fi
}%
\hfill\discretionary{}{}{}%

\vskip2pt\hrule\vskip2pt

\leavevmode\setbox\matricesbox=\hbox{%
{$\left[\!\llap{\phantom{%
\begingroup \smaller\smaller\smaller\begin{tabular}{@{}c@{}}%
\phantom{0}\\\phantom{0}\\\phantom{0}
\end{tabular}\endgroup%
}}\right.$}%
\begingroup \smaller\smaller\smaller\begin{tabular}{@{}c@{}}%
-1/4\\\phantom{0}\\\phantom{0}
\end{tabular}\endgroup%
\kern3pt%
\begingroup \smaller\smaller\smaller\begin{tabular}{@{}c@{}}%
\phantom{0}\\3\\\phantom{0}
\end{tabular}\endgroup%
\kern3pt%
\begingroup \smaller\smaller\smaller\begin{tabular}{@{}c@{}}%
\phantom{0}\\\phantom{0}\\3
\end{tabular}\endgroup%
{$\left.\llap{\phantom{%
\begingroup \smaller\smaller\smaller\begin{tabular}{@{}c@{}}%
\phantom{0}\\\phantom{0}\\\phantom{0}
\end{tabular}\endgroup%
}}\!\right]$}%
{$\left[\!\llap{\phantom{%
\begingroup \smaller\smaller\smaller\begin{tabular}{@{}c@{}}%
0\\0\\0
\end{tabular}\endgroup%
}}\right.$}%
\begingroup \smaller\smaller\smaller\begin{tabular}{@{}c@{}}%
2\\1\\0
\end{tabular}\endgroup%
{$\left.\llap{\phantom{%
\begingroup \smaller\smaller\smaller\begin{tabular}{@{}c@{}}%
0\\0\\0
\end{tabular}\endgroup%
}}\!\right]$}%
}%
\ifdim\wd\matricesbox>\halfwidth\myboxwidth=\hsize\else\myboxwidth=\halfwidth\fi
\vbox{%
\ifdim\myboxwidth=\hsize
\setbox\onelinebox=\hbox{%
\vbox{\hbox{%
$\Pi_{4,1}=B_2=\hbox{GN}_{29}$ spans $L_{3.4}$%
}\hbox{%
$|\slashthree|\slashthree|\slashthree|\slashthree\rtimes D_{8}$%
}%
}%
\hfill\copy\matricesbox
}%
\ifdim\wd\onelinebox>\myboxwidth
\hbox to \myboxwidth{%
$\Pi_{4,1}=B_2=\hbox{GN}_{29}$ spans $L_{3.4}$%
\hfil
$|\slashthree|\slashthree|\slashthree|\slashthree\rtimes D_{8}$%
}%
\box\matricesbox
\else
\hbox to \myboxwidth{%
\unhbox\onelinebox
}%
\fi
\else
\hbox to \myboxwidth{%
$\Pi_{4,1}=B_2=\hbox{GN}_{29}$ spans $L_{3.4}$%
\hfil}%
\hbox to \myboxwidth{%
$|\slashthree|\slashthree|\slashthree|\slashthree\rtimes D_{8}$%
\hfil}%
\box\matricesbox
\fi
}%
\hfill\discretionary{}{}{}%
\setbox\matricesbox=\hbox{%
{$\left[\!\llap{\phantom{%
\begingroup \smaller\smaller\smaller\begin{tabular}{@{}c@{}}%
\phantom{0}\\\phantom{0}\\\phantom{0}
\end{tabular}\endgroup%
}}\right.$}%
\begingroup \smaller\smaller\smaller\begin{tabular}{@{}c@{}}%
-1\\\phantom{0}\\\phantom{0}
\end{tabular}\endgroup%
\kern3pt%
\begingroup \smaller\smaller\smaller\begin{tabular}{@{}c@{}}%
\phantom{0}\\2\\\phantom{0}
\end{tabular}\endgroup%
\kern3pt%
\begingroup \smaller\smaller\smaller\begin{tabular}{@{}c@{}}%
\phantom{0}\\\phantom{0}\\2
\end{tabular}\endgroup%
{$\left.\llap{\phantom{%
\begingroup \smaller\smaller\smaller\begin{tabular}{@{}c@{}}%
\phantom{0}\\\phantom{0}\\\phantom{0}
\end{tabular}\endgroup%
}}\!\right]$}%
{$\left[\!\llap{\phantom{%
\begingroup \smaller\smaller\smaller\begin{tabular}{@{}c@{}}%
0\\0\\0
\end{tabular}\endgroup%
}}\right.$}%
\begingroup \smaller\smaller\smaller\begin{tabular}{@{}c@{}}%
1\\0\\-1
\end{tabular}\endgroup%
{$\left.\llap{\phantom{%
\begingroup \smaller\smaller\smaller\begin{tabular}{@{}c@{}}%
0\\0\\0
\end{tabular}\endgroup%
}}\!\right]$}%
}%
\ifdim\wd\matricesbox>\halfwidth\myboxwidth=\hsize\else\myboxwidth=\halfwidth\fi
\vbox{%
\ifdim\myboxwidth=\hsize
\setbox\onelinebox=\hbox{%
\vbox{\hbox{%
$\Pi_{4,2}=A_{2,II}=\hbox{GN}_{35}$ spans $L_{1.9}$%
}\hbox{%
$|\slashinfty|\slashinfty|\slashinfty|\slashinfty\rtimes D_{8}$ (shared)%
}%
}%
\hfill\copy\matricesbox
}%
\ifdim\wd\onelinebox>\myboxwidth
\hbox to \myboxwidth{%
$\Pi_{4,2}=A_{2,II}=\hbox{GN}_{35}$ spans $L_{1.9}$%
\hfil
$|\slashinfty|\slashinfty|\slashinfty|\slashinfty\rtimes D_{8}$ (shared)%
}%
\box\matricesbox
\else
\hbox to \myboxwidth{%
\unhbox\onelinebox
}%
\fi
\else
\hbox to \myboxwidth{%
$\Pi_{4,2}=A_{2,II}=\hbox{GN}_{35}$ spans $L_{1.9}$%
\hfil}%
\hbox to \myboxwidth{%
$|\slashinfty|\slashinfty|\slashinfty|\slashinfty\rtimes D_{8}$ (shared)%
\hfil}%
\box\matricesbox
\fi
}%
\hfill\discretionary{}{}{}%
\setbox\matricesbox=\hbox{%
{$\left[\!\llap{\phantom{%
\begingroup \smaller\smaller\smaller\begin{tabular}{@{}c@{}}%
\phantom{0}\\\phantom{0}\\\phantom{0}
\end{tabular}\endgroup%
}}\right.$}%
\begingroup \smaller\smaller\smaller\begin{tabular}{@{}c@{}}%
-3/8\\\phantom{0}\\\phantom{0}
\end{tabular}\endgroup%
\kern3pt%
\begingroup \smaller\smaller\smaller\begin{tabular}{@{}c@{}}%
\phantom{0}\\2\\\phantom{0}
\end{tabular}\endgroup%
\kern3pt%
\begingroup \smaller\smaller\smaller\begin{tabular}{@{}c@{}}%
\phantom{0}\\\phantom{0}\\3/2
\end{tabular}\endgroup%
{$\left.\llap{\phantom{%
\begingroup \smaller\smaller\smaller\begin{tabular}{@{}c@{}}%
\phantom{0}\\\phantom{0}\\\phantom{0}
\end{tabular}\endgroup%
}}\!\right]$}%
{$\left[\!\llap{\phantom{%
\begingroup \smaller\smaller\smaller\begin{tabular}{@{}c@{}}%
0\\0\\0
\end{tabular}\endgroup%
}}\right.$}%
\begingroup \smaller\smaller\smaller\begin{tabular}{@{}c@{}}%
2\\-1\\1
\end{tabular}\endgroup%
{$\left.\llap{\phantom{%
\begingroup \smaller\smaller\smaller\begin{tabular}{@{}c@{}}%
0\\0\\0
\end{tabular}\endgroup%
}}\!\right]$}%
}%
\ifdim\wd\matricesbox>\halfwidth\myboxwidth=\hsize\else\myboxwidth=\halfwidth\fi
\vbox{%
\ifdim\myboxwidth=\hsize
\setbox\onelinebox=\hbox{%
\vbox{\hbox{%
$\Pi_{4,3}=A_{3,I}=\hbox{GN}_{33}$ spans $L_{7.9}$%
}\hbox{%
$\slashthree\slashinfty\slashthree\slashinfty\rtimes D_{4}$%
}%
}%
\hfill\copy\matricesbox
}%
\ifdim\wd\onelinebox>\myboxwidth
\hbox to \myboxwidth{%
$\Pi_{4,3}=A_{3,I}=\hbox{GN}_{33}$ spans $L_{7.9}$%
\hfil
$\slashthree\slashinfty\slashthree\slashinfty\rtimes D_{4}$%
}%
\box\matricesbox
\else
\hbox to \myboxwidth{%
\unhbox\onelinebox
}%
\fi
\else
\hbox to \myboxwidth{%
$\Pi_{4,3}=A_{3,I}=\hbox{GN}_{33}$ spans $L_{7.9}$%
\hfil}%
\hbox to \myboxwidth{%
$\slashthree\slashinfty\slashthree\slashinfty\rtimes D_{4}$%
\hfil}%
\box\matricesbox
\fi
}%
\hfill\discretionary{}{}{}%
\setbox\matricesbox=\hbox{%
{$\left[\!\llap{\phantom{%
\begingroup \smaller\smaller\smaller\begin{tabular}{@{}c@{}}%
\phantom{0}\\\phantom{0}\\\phantom{0}
\end{tabular}\endgroup%
}}\right.$}%
\begingroup \smaller\smaller\smaller\begin{tabular}{@{}c@{}}%
-1/2\\\phantom{0}\\\phantom{0}
\end{tabular}\endgroup%
\kern3pt%
\begingroup \smaller\smaller\smaller\begin{tabular}{@{}c@{}}%
\phantom{0}\\3/2\\\phantom{0}
\end{tabular}\endgroup%
\kern3pt%
\begingroup \smaller\smaller\smaller\begin{tabular}{@{}c@{}}%
\phantom{0}\\\phantom{0}\\4
\end{tabular}\endgroup%
{$\left.\llap{\phantom{%
\begingroup \smaller\smaller\smaller\begin{tabular}{@{}c@{}}%
\phantom{0}\\\phantom{0}\\\phantom{0}
\end{tabular}\endgroup%
}}\!\right]$}%
{$\left[\!\llap{\phantom{%
\begingroup \smaller\smaller\smaller\begin{tabular}{@{}c@{}}%
0\\0\\0
\end{tabular}\endgroup%
}}\right.$}%
\begingroup \smaller\smaller\smaller\begin{tabular}{@{}c@{}}%
1\\-1\\0
\end{tabular}\endgroup%
\kern3pt%
\begingroup \smaller\smaller\smaller\begin{tabular}{@{}c@{}}%
2\\0\\-1
\end{tabular}\endgroup%
{$\left.\llap{\phantom{%
\begingroup \smaller\smaller\smaller\begin{tabular}{@{}c@{}}%
0\\0\\0
\end{tabular}\endgroup%
}}\!\right]$}%
}%
\ifdim\wd\matricesbox>\halfwidth\myboxwidth=\hsize\else\myboxwidth=\halfwidth\fi
\vbox{%
\ifdim\myboxwidth=\hsize
\setbox\onelinebox=\hbox{%
\vbox{\hbox{%
$\Pi_{4,4}$ spans $L_{123.4}$%
}\hbox{%
$|4|4|4|4\rtimes D_{4}$%
}%
}%
\hfill\copy\matricesbox
}%
\ifdim\wd\onelinebox>\myboxwidth
\hbox to \myboxwidth{%
$\Pi_{4,4}$ spans $L_{123.4}$%
\hfil
$|4|4|4|4\rtimes D_{4}$%
}%
\box\matricesbox
\else
\hbox to \myboxwidth{%
\unhbox\onelinebox
}%
\fi
\else
\hbox to \myboxwidth{%
$\Pi_{4,4}$ spans $L_{123.4}$%
\hfil}%
\hbox to \myboxwidth{%
$|4|4|4|4\rtimes D_{4}$%
\hfil}%
\box\matricesbox
\fi
}%
\hfill\discretionary{}{}{}%
\setbox\matricesbox=\hbox{%
{$\left[\!\llap{\phantom{%
\begingroup \smaller\smaller\smaller\begin{tabular}{@{}c@{}}%
\phantom{0}\\\phantom{0}\\\phantom{0}
\end{tabular}\endgroup%
}}\right.$}%
\begingroup \smaller\smaller\smaller\begin{tabular}{@{}c@{}}%
-1/4\\\phantom{0}\\\phantom{0}
\end{tabular}\endgroup%
\kern3pt%
\begingroup \smaller\smaller\smaller\begin{tabular}{@{}c@{}}%
\phantom{0}\\3\\\phantom{0}
\end{tabular}\endgroup%
\kern3pt%
\begingroup \smaller\smaller\smaller\begin{tabular}{@{}c@{}}%
\phantom{0}\\\phantom{0}\\15
\end{tabular}\endgroup%
{$\left.\llap{\phantom{%
\begingroup \smaller\smaller\smaller\begin{tabular}{@{}c@{}}%
\phantom{0}\\\phantom{0}\\\phantom{0}
\end{tabular}\endgroup%
}}\!\right]$}%
{$\left[\!\llap{\phantom{%
\begingroup \smaller\smaller\smaller\begin{tabular}{@{}c@{}}%
0\\0\\0
\end{tabular}\endgroup%
}}\right.$}%
\begingroup \smaller\smaller\smaller\begin{tabular}{@{}c@{}}%
2\\1\\0
\end{tabular}\endgroup%
\kern3pt%
\begingroup \smaller\smaller\smaller\begin{tabular}{@{}c@{}}%
6\\0\\1
\end{tabular}\endgroup%
{$\left.\llap{\phantom{%
\begingroup \smaller\smaller\smaller\begin{tabular}{@{}c@{}}%
0\\0\\0
\end{tabular}\endgroup%
}}\!\right]$}%
}%
\ifdim\wd\matricesbox>\halfwidth\myboxwidth=\hsize\else\myboxwidth=\halfwidth\fi
\vbox{%
\ifdim\myboxwidth=\hsize
\setbox\onelinebox=\hbox{%
\vbox{\hbox{%
$\Pi_{4,5}$ spans $L_{16.8}$%
}\hbox{%
$|6|6|6|6\rtimes D_{4}$%
}%
}%
\hfill\copy\matricesbox
}%
\ifdim\wd\onelinebox>\myboxwidth
\hbox to \myboxwidth{%
$\Pi_{4,5}$ spans $L_{16.8}$%
\hfil
$|6|6|6|6\rtimes D_{4}$%
}%
\box\matricesbox
\else
\hbox to \myboxwidth{%
\unhbox\onelinebox
}%
\fi
\else
\hbox to \myboxwidth{%
$\Pi_{4,5}$ spans $L_{16.8}$%
\hfil}%
\hbox to \myboxwidth{%
$|6|6|6|6\rtimes D_{4}$%
\hfil}%
\box\matricesbox
\fi
}%
\hfill\discretionary{}{}{}%
\setbox\matricesbox=\hbox{%
{$\left[\!\llap{\phantom{%
\begingroup \smaller\smaller\smaller\begin{tabular}{@{}c@{}}%
\phantom{0}\\\phantom{0}\\\phantom{0}
\end{tabular}\endgroup%
}}\right.$}%
\begingroup \smaller\smaller\smaller\begin{tabular}{@{}c@{}}%
-1/2\\\phantom{0}\\\phantom{0}
\end{tabular}\endgroup%
\kern3pt%
\begingroup \smaller\smaller\smaller\begin{tabular}{@{}c@{}}%
\phantom{0}\\3/2\\\phantom{0}
\end{tabular}\endgroup%
\kern3pt%
\begingroup \smaller\smaller\smaller\begin{tabular}{@{}c@{}}%
\phantom{0}\\\phantom{0}\\12
\end{tabular}\endgroup%
{$\left.\llap{\phantom{%
\begingroup \smaller\smaller\smaller\begin{tabular}{@{}c@{}}%
\phantom{0}\\\phantom{0}\\\phantom{0}
\end{tabular}\endgroup%
}}\!\right]$}%
{$\left[\!\llap{\phantom{%
\begingroup \smaller\smaller\smaller\begin{tabular}{@{}c@{}}%
0\\0\\0
\end{tabular}\endgroup%
}}\right.$}%
\begingroup \smaller\smaller\smaller\begin{tabular}{@{}c@{}}%
1\\-1\\0
\end{tabular}\endgroup%
\kern3pt%
\begingroup \smaller\smaller\smaller\begin{tabular}{@{}c@{}}%
4\\0\\-1
\end{tabular}\endgroup%
{$\left.\llap{\phantom{%
\begingroup \smaller\smaller\smaller\begin{tabular}{@{}c@{}}%
0\\0\\0
\end{tabular}\endgroup%
}}\!\right]$}%
}%
\ifdim\wd\matricesbox>\halfwidth\myboxwidth=\hsize\else\myboxwidth=\halfwidth\fi
\vbox{%
\ifdim\myboxwidth=\hsize
\setbox\onelinebox=\hbox{%
\vbox{\hbox{%
$\Pi_{4,6}=\hbox{GN}_{30}$ spans $L_{4.17}$%
}\hbox{%
$|\infty|\infty|\infty|\infty\rtimes D_{4}$%
}%
}%
\hfill\copy\matricesbox
}%
\ifdim\wd\onelinebox>\myboxwidth
\hbox to \myboxwidth{%
$\Pi_{4,6}=\hbox{GN}_{30}$ spans $L_{4.17}$%
\hfil
$|\infty|\infty|\infty|\infty\rtimes D_{4}$%
}%
\box\matricesbox
\else
\hbox to \myboxwidth{%
\unhbox\onelinebox
}%
\fi
\else
\hbox to \myboxwidth{%
$\Pi_{4,6}=\hbox{GN}_{30}$ spans $L_{4.17}$%
\hfil}%
\hbox to \myboxwidth{%
$|\infty|\infty|\infty|\infty\rtimes D_{4}$%
\hfil}%
\box\matricesbox
\fi
}%
\hfill\discretionary{}{}{}%
\setbox\matricesbox=\hbox{%
{$\left[\!\llap{\phantom{%
\begingroup \smaller\smaller\smaller\begin{tabular}{@{}c@{}}%
\phantom{0}\\\phantom{0}\\\phantom{0}
\end{tabular}\endgroup%
}}\right.$}%
\begingroup \smaller\smaller\smaller\begin{tabular}{@{}c@{}}%
-1/8\\\phantom{0}\\\phantom{0}
\end{tabular}\endgroup%
\kern3pt%
\begingroup \smaller\smaller\smaller\begin{tabular}{@{}c@{}}%
\phantom{0}\\1\\\phantom{0}
\end{tabular}\endgroup%
\kern3pt%
\begingroup \smaller\smaller\smaller\begin{tabular}{@{}c@{}}%
\phantom{0}\\\phantom{0}\\3/2
\end{tabular}\endgroup%
{$\left.\llap{\phantom{%
\begingroup \smaller\smaller\smaller\begin{tabular}{@{}c@{}}%
\phantom{0}\\\phantom{0}\\\phantom{0}
\end{tabular}\endgroup%
}}\!\right]$}%
{$\left[\!\llap{\phantom{%
\begingroup \smaller\smaller\smaller\begin{tabular}{@{}c@{}}%
0\\0\\0
\end{tabular}\endgroup%
}}\right.$}%
\begingroup \smaller\smaller\smaller\begin{tabular}{@{}c@{}}%
2\\-1\\-1
\end{tabular}\endgroup%
{$\left.\llap{\phantom{%
\begingroup \smaller\smaller\smaller\begin{tabular}{@{}c@{}}%
0\\0\\0
\end{tabular}\endgroup%
}}\!\right]$}%
}%
\ifdim\wd\matricesbox>\halfwidth\myboxwidth=\hsize\else\myboxwidth=\halfwidth\fi
\vbox{%
\ifdim\myboxwidth=\hsize
\setbox\onelinebox=\hbox{%
\vbox{\hbox{%
$\Pi_{4,7}=B_1=\hbox{GN}_{21}$ spans $L_{3.1}$%
}\hbox{%
$\slashthree\slashtwo\slashthree\slashtwo\rtimes D_{4}$ (shared)%
}%
}%
\hfill\copy\matricesbox
}%
\ifdim\wd\onelinebox>\myboxwidth
\hbox to \myboxwidth{%
$\Pi_{4,7}=B_1=\hbox{GN}_{21}$ spans $L_{3.1}$%
\hfil
$\slashthree\slashtwo\slashthree\slashtwo\rtimes D_{4}$ (shared)%
}%
\box\matricesbox
\else
\hbox to \myboxwidth{%
\unhbox\onelinebox
}%
\fi
\else
\hbox to \myboxwidth{%
$\Pi_{4,7}=B_1=\hbox{GN}_{21}$ spans $L_{3.1}$%
\hfil}%
\hbox to \myboxwidth{%
$\slashthree\slashtwo\slashthree\slashtwo\rtimes D_{4}$ (shared)%
\hfil}%
\box\matricesbox
\fi
}%
\hfill\discretionary{}{}{}%
\setbox\matricesbox=\hbox{%
{$\left[\!\llap{\phantom{%
\begingroup \smaller\smaller\smaller\begin{tabular}{@{}c@{}}%
\phantom{0}\\\phantom{0}\\\phantom{0}
\end{tabular}\endgroup%
}}\right.$}%
\begingroup \smaller\smaller\smaller\begin{tabular}{@{}c@{}}%
-1/2\\\phantom{0}\\\phantom{0}
\end{tabular}\endgroup%
\kern3pt%
\begingroup \smaller\smaller\smaller\begin{tabular}{@{}c@{}}%
\phantom{0}\\1/2\\\phantom{0}
\end{tabular}\endgroup%
\kern3pt%
\begingroup \smaller\smaller\smaller\begin{tabular}{@{}c@{}}%
\phantom{0}\\\phantom{0}\\1
\end{tabular}\endgroup%
{$\left.\llap{\phantom{%
\begingroup \smaller\smaller\smaller\begin{tabular}{@{}c@{}}%
\phantom{0}\\\phantom{0}\\\phantom{0}
\end{tabular}\endgroup%
}}\!\right]$}%
{$\left[\!\llap{\phantom{%
\begingroup \smaller\smaller\smaller\begin{tabular}{@{}c@{}}%
0\\0\\0
\end{tabular}\endgroup%
}}\right.$}%
\begingroup \smaller\smaller\smaller\begin{tabular}{@{}c@{}}%
1\\-1\\-1
\end{tabular}\endgroup%
{$\left.\llap{\phantom{%
\begingroup \smaller\smaller\smaller\begin{tabular}{@{}c@{}}%
0\\0\\0
\end{tabular}\endgroup%
}}\!\right]$}%
}%
\ifdim\wd\matricesbox>\halfwidth\myboxwidth=\hsize\else\myboxwidth=\halfwidth\fi
\vbox{%
\ifdim\myboxwidth=\hsize
\setbox\onelinebox=\hbox{%
\vbox{\hbox{%
$\Pi_{4,8}=A_{2,I}=\hbox{GN}_{27}$ spans $L_{1.6}$%
}\hbox{%
$\slashinfty\slashtwo\slashinfty\slashtwo\rtimes D_{4}$ (shared)%
}%
}%
\hfill\copy\matricesbox
}%
\ifdim\wd\onelinebox>\myboxwidth
\hbox to \myboxwidth{%
$\Pi_{4,8}=A_{2,I}=\hbox{GN}_{27}$ spans $L_{1.6}$%
\hfil
$\slashinfty\slashtwo\slashinfty\slashtwo\rtimes D_{4}$ (shared)%
}%
\box\matricesbox
\else
\hbox to \myboxwidth{%
\unhbox\onelinebox
}%
\fi
\else
\hbox to \myboxwidth{%
$\Pi_{4,8}=A_{2,I}=\hbox{GN}_{27}$ spans $L_{1.6}$%
\hfil}%
\hbox to \myboxwidth{%
$\slashinfty\slashtwo\slashinfty\slashtwo\rtimes D_{4}$ (shared)%
\hfil}%
\box\matricesbox
\fi
}%
\hfill\discretionary{}{}{}%
\setbox\matricesbox=\hbox{%
{$\left[\!\llap{\phantom{%
\begingroup \smaller\smaller\smaller\begin{tabular}{@{}c@{}}%
\phantom{0}\\\phantom{0}\\\phantom{0}
\end{tabular}\endgroup%
}}\right.$}%
\begingroup \smaller\smaller\smaller\begin{tabular}{@{}c@{}}%
-1/7\\\phantom{0}\\\phantom{0}
\end{tabular}\endgroup%
\kern3pt%
\begingroup \smaller\smaller\smaller\begin{tabular}{@{}c@{}}%
\phantom{0}\\1/14\\\phantom{0}
\end{tabular}\endgroup%
\kern3pt%
\begingroup \smaller\smaller\smaller\begin{tabular}{@{}c@{}}%
\phantom{0}\\\phantom{0}\\5/2
\end{tabular}\endgroup%
{$\left.\llap{\phantom{%
\begingroup \smaller\smaller\smaller\begin{tabular}{@{}c@{}}%
\phantom{0}\\\phantom{0}\\\phantom{0}
\end{tabular}\endgroup%
}}\!\right]$}%
{$\left[\!\llap{\phantom{%
\begingroup \smaller\smaller\smaller\begin{tabular}{@{}c@{}}%
0\\0\\0
\end{tabular}\endgroup%
}}\right.$}%
\begingroup \smaller\smaller\smaller\begin{tabular}{@{}c@{}}%
2\\-6\\0
\end{tabular}\endgroup%
\kern3pt%
\begingroup \smaller\smaller\smaller\begin{tabular}{@{}c@{}}%
2\\1\\1
\end{tabular}\endgroup%
\kern3pt%
\begingroup \smaller\smaller\smaller\begin{tabular}{@{}c@{}}%
1\\4\\0
\end{tabular}\endgroup%
{$\left.\llap{\phantom{%
\begingroup \smaller\smaller\smaller\begin{tabular}{@{}c@{}}%
0\\0\\0
\end{tabular}\endgroup%
}}\!\right]$}%
}%
\ifdim\wd\matricesbox>\halfwidth\myboxwidth=\hsize\else\myboxwidth=\halfwidth\fi
\vbox{%
\ifdim\myboxwidth=\hsize
\setbox\onelinebox=\hbox{%
\vbox{\hbox{%
$\Pi_{4,9}$ spans $L_{6.2}$%
}\hbox{%
$3|32|2\rtimes D_{2}$%
}%
}%
\hfill\copy\matricesbox
}%
\ifdim\wd\onelinebox>\myboxwidth
\hbox to \myboxwidth{%
$\Pi_{4,9}$ spans $L_{6.2}$%
\hfil
$3|32|2\rtimes D_{2}$%
}%
\box\matricesbox
\else
\hbox to \myboxwidth{%
\unhbox\onelinebox
}%
\fi
\else
\hbox to \myboxwidth{%
$\Pi_{4,9}$ spans $L_{6.2}$%
\hfil}%
\hbox to \myboxwidth{%
$3|32|2\rtimes D_{2}$%
\hfil}%
\box\matricesbox
\fi
}%
\hfill\discretionary{}{}{}%
\setbox\matricesbox=\hbox{%
{$\left[\!\llap{\phantom{%
\begingroup \smaller\smaller\smaller\begin{tabular}{@{}c@{}}%
\phantom{0}\\\phantom{0}\\\phantom{0}
\end{tabular}\endgroup%
}}\right.$}%
\begingroup \smaller\smaller\smaller\begin{tabular}{@{}c@{}}%
-6/25\\\phantom{0}\\\phantom{0}
\end{tabular}\endgroup%
\kern3pt%
\begingroup \smaller\smaller\smaller\begin{tabular}{@{}c@{}}%
\phantom{0}\\3/50\\\phantom{0}
\end{tabular}\endgroup%
\kern3pt%
\begingroup \smaller\smaller\smaller\begin{tabular}{@{}c@{}}%
\phantom{0}\\\phantom{0}\\1/2
\end{tabular}\endgroup%
{$\left.\llap{\phantom{%
\begingroup \smaller\smaller\smaller\begin{tabular}{@{}c@{}}%
\phantom{0}\\\phantom{0}\\\phantom{0}
\end{tabular}\endgroup%
}}\!\right]$}%
{$\left[\!\llap{\phantom{%
\begingroup \smaller\smaller\smaller\begin{tabular}{@{}c@{}}%
0\\0\\0
\end{tabular}\endgroup%
}}\right.$}%
\begingroup \smaller\smaller\smaller\begin{tabular}{@{}c@{}}%
6\\13\\3
\end{tabular}\endgroup%
\kern3pt%
\begingroup \smaller\smaller\smaller\begin{tabular}{@{}c@{}}%
2\\-4\\2
\end{tabular}\endgroup%
{$\left.\llap{\phantom{%
\begingroup \smaller\smaller\smaller\begin{tabular}{@{}c@{}}%
0\\0\\0
\end{tabular}\endgroup%
}}\!\right]$}%
}%
\ifdim\wd\matricesbox>\halfwidth\myboxwidth=\hsize\else\myboxwidth=\halfwidth\fi
\vbox{%
\ifdim\myboxwidth=\hsize
\setbox\onelinebox=\hbox{%
\vbox{\hbox{%
$\Pi_{4,10}$ spans $L_{7.9}$%
}\hbox{%
$\slashthree6\slashinfty6\rtimes D_{2}$ (shared)%
}%
}%
\hfill\copy\matricesbox
}%
\ifdim\wd\onelinebox>\myboxwidth
\hbox to \myboxwidth{%
$\Pi_{4,10}$ spans $L_{7.9}$%
\hfil
$\slashthree6\slashinfty6\rtimes D_{2}$ (shared)%
}%
\box\matricesbox
\else
\hbox to \myboxwidth{%
\unhbox\onelinebox
}%
\fi
\else
\hbox to \myboxwidth{%
$\Pi_{4,10}$ spans $L_{7.9}$%
\hfil}%
\hbox to \myboxwidth{%
$\slashthree6\slashinfty6\rtimes D_{2}$ (shared)%
\hfil}%
\box\matricesbox
\fi
}%
\hfill\discretionary{}{}{}%
\setbox\matricesbox=\hbox{%
{$\left[\!\llap{\phantom{%
\begingroup \smaller\smaller\smaller\begin{tabular}{@{}c@{}}%
\phantom{0}\\\phantom{0}\\\phantom{0}
\end{tabular}\endgroup%
}}\right.$}%
\begingroup \smaller\smaller\smaller\begin{tabular}{@{}c@{}}%
-2/9\\\phantom{0}\\\phantom{0}
\end{tabular}\endgroup%
\kern3pt%
\begingroup \smaller\smaller\smaller\begin{tabular}{@{}c@{}}%
\phantom{0}\\1/18\\\phantom{0}
\end{tabular}\endgroup%
\kern3pt%
\begingroup \smaller\smaller\smaller\begin{tabular}{@{}c@{}}%
\phantom{0}\\\phantom{0}\\3/2
\end{tabular}\endgroup%
{$\left.\llap{\phantom{%
\begingroup \smaller\smaller\smaller\begin{tabular}{@{}c@{}}%
\phantom{0}\\\phantom{0}\\\phantom{0}
\end{tabular}\endgroup%
}}\!\right]$}%
{$\left[\!\llap{\phantom{%
\begingroup \smaller\smaller\smaller\begin{tabular}{@{}c@{}}%
0\\0\\0
\end{tabular}\endgroup%
}}\right.$}%
\begingroup \smaller\smaller\smaller\begin{tabular}{@{}c@{}}%
2\\-5\\1
\end{tabular}\endgroup%
\kern3pt%
\begingroup \smaller\smaller\smaller\begin{tabular}{@{}c@{}}%
6\\12\\2
\end{tabular}\endgroup%
{$\left.\llap{\phantom{%
\begingroup \smaller\smaller\smaller\begin{tabular}{@{}c@{}}%
0\\0\\0
\end{tabular}\endgroup%
}}\!\right]$}%
}%
\ifdim\wd\matricesbox>\halfwidth\myboxwidth=\hsize\else\myboxwidth=\halfwidth\fi
\vbox{%
\ifdim\myboxwidth=\hsize
\setbox\onelinebox=\hbox{%
\vbox{\hbox{%
$\Pi_{4,11}$ spans $L_{7.13}$%
}\hbox{%
$\slashthree6\slashinfty6\rtimes D_{2}$ (shared)%
}%
}%
\hfill\copy\matricesbox
}%
\ifdim\wd\onelinebox>\myboxwidth
\hbox to \myboxwidth{%
$\Pi_{4,11}$ spans $L_{7.13}$%
\hfil
$\slashthree6\slashinfty6\rtimes D_{2}$ (shared)%
}%
\box\matricesbox
\else
\hbox to \myboxwidth{%
\unhbox\onelinebox
}%
\fi
\else
\hbox to \myboxwidth{%
$\Pi_{4,11}$ spans $L_{7.13}$%
\hfil}%
\hbox to \myboxwidth{%
$\slashthree6\slashinfty6\rtimes D_{2}$ (shared)%
\hfil}%
\box\matricesbox
\fi
}%
\hfill\discretionary{}{}{}%
\setbox\matricesbox=\hbox{%
{$\left[\!\llap{\phantom{%
\begingroup \smaller\smaller\smaller\begin{tabular}{@{}c@{}}%
\phantom{0}\\\phantom{0}\\\phantom{0}
\end{tabular}\endgroup%
}}\right.$}%
\begingroup \smaller\smaller\smaller\begin{tabular}{@{}c@{}}%
-1/20\\\phantom{0}\\\phantom{0}
\end{tabular}\endgroup%
\kern3pt%
\begingroup \smaller\smaller\smaller\begin{tabular}{@{}c@{}}%
\phantom{0}\\3/10\\\phantom{0}
\end{tabular}\endgroup%
\kern3pt%
\begingroup \smaller\smaller\smaller\begin{tabular}{@{}c@{}}%
\phantom{0}\\\phantom{0}\\15/2
\end{tabular}\endgroup%
{$\left.\llap{\phantom{%
\begingroup \smaller\smaller\smaller\begin{tabular}{@{}c@{}}%
\phantom{0}\\\phantom{0}\\\phantom{0}
\end{tabular}\endgroup%
}}\!\right]$}%
{$\left[\!\llap{\phantom{%
\begingroup \smaller\smaller\smaller\begin{tabular}{@{}c@{}}%
0\\0\\0
\end{tabular}\endgroup%
}}\right.$}%
\begingroup \smaller\smaller\smaller\begin{tabular}{@{}c@{}}%
12\\8\\0
\end{tabular}\endgroup%
\kern3pt%
\begingroup \smaller\smaller\smaller\begin{tabular}{@{}c@{}}%
6\\-1\\1
\end{tabular}\endgroup%
\kern3pt%
\begingroup \smaller\smaller\smaller\begin{tabular}{@{}c@{}}%
4\\-4\\0
\end{tabular}\endgroup%
{$\left.\llap{\phantom{%
\begingroup \smaller\smaller\smaller\begin{tabular}{@{}c@{}}%
0\\0\\0
\end{tabular}\endgroup%
}}\!\right]$}%
}%
\ifdim\wd\matricesbox>\halfwidth\myboxwidth=\hsize\else\myboxwidth=\halfwidth\fi
\vbox{%
\ifdim\myboxwidth=\hsize
\setbox\onelinebox=\hbox{%
\vbox{\hbox{%
$\Pi_{4,12}$ spans $L_{19.5}$%
}\hbox{%
$4|42|2\rtimes D_{2}$ (shared)%
}%
}%
\hfill\copy\matricesbox
}%
\ifdim\wd\onelinebox>\myboxwidth
\hbox to \myboxwidth{%
$\Pi_{4,12}$ spans $L_{19.5}$%
\hfil
$4|42|2\rtimes D_{2}$ (shared)%
}%
\box\matricesbox
\else
\hbox to \myboxwidth{%
\unhbox\onelinebox
}%
\fi
\else
\hbox to \myboxwidth{%
$\Pi_{4,12}$ spans $L_{19.5}$%
\hfil}%
\hbox to \myboxwidth{%
$4|42|2\rtimes D_{2}$ (shared)%
\hfil}%
\box\matricesbox
\fi
}%
\hfill\discretionary{}{}{}%
\setbox\matricesbox=\hbox{%
{$\left[\!\llap{\phantom{%
\begingroup \smaller\smaller\smaller\begin{tabular}{@{}c@{}}%
\phantom{0}\\\phantom{0}\\\phantom{0}
\end{tabular}\endgroup%
}}\right.$}%
\begingroup \smaller\smaller\smaller\begin{tabular}{@{}c@{}}%
-2/5\\\phantom{0}\\\phantom{0}
\end{tabular}\endgroup%
\kern3pt%
\begingroup \smaller\smaller\smaller\begin{tabular}{@{}c@{}}%
\phantom{0}\\1/10\\\phantom{0}
\end{tabular}\endgroup%
\kern3pt%
\begingroup \smaller\smaller\smaller\begin{tabular}{@{}c@{}}%
\phantom{0}\\\phantom{0}\\1/2
\end{tabular}\endgroup%
{$\left.\llap{\phantom{%
\begingroup \smaller\smaller\smaller\begin{tabular}{@{}c@{}}%
\phantom{0}\\\phantom{0}\\\phantom{0}
\end{tabular}\endgroup%
}}\!\right]$}%
{$\left[\!\llap{\phantom{%
\begingroup \smaller\smaller\smaller\begin{tabular}{@{}c@{}}%
0\\0\\0
\end{tabular}\endgroup%
}}\right.$}%
\begingroup \smaller\smaller\smaller\begin{tabular}{@{}c@{}}%
2\\4\\-2
\end{tabular}\endgroup%
\kern3pt%
\begingroup \smaller\smaller\smaller\begin{tabular}{@{}c@{}}%
1\\-3\\-1
\end{tabular}\endgroup%
{$\left.\llap{\phantom{%
\begingroup \smaller\smaller\smaller\begin{tabular}{@{}c@{}}%
0\\0\\0
\end{tabular}\endgroup%
}}\!\right]$}%
}%
\ifdim\wd\matricesbox>\halfwidth\myboxwidth=\hsize\else\myboxwidth=\halfwidth\fi
\vbox{%
\ifdim\myboxwidth=\hsize
\setbox\onelinebox=\hbox{%
\vbox{\hbox{%
$\Pi_{4,13}$ spans $L_{1.3}$%
}\hbox{%
$4\slashinfty4\slashtwo\rtimes D_{2}$ (shared)%
}%
}%
\hfill\copy\matricesbox
}%
\ifdim\wd\onelinebox>\myboxwidth
\hbox to \myboxwidth{%
$\Pi_{4,13}$ spans $L_{1.3}$%
\hfil
$4\slashinfty4\slashtwo\rtimes D_{2}$ (shared)%
}%
\box\matricesbox
\else
\hbox to \myboxwidth{%
\unhbox\onelinebox
}%
\fi
\else
\hbox to \myboxwidth{%
$\Pi_{4,13}$ spans $L_{1.3}$%
\hfil}%
\hbox to \myboxwidth{%
$4\slashinfty4\slashtwo\rtimes D_{2}$ (shared)%
\hfil}%
\box\matricesbox
\fi
}%
\hfill\discretionary{}{}{}%
\setbox\matricesbox=\hbox{%
{$\left[\!\llap{\phantom{%
\begingroup \smaller\smaller\smaller\begin{tabular}{@{}c@{}}%
\phantom{0}\\\phantom{0}\\\phantom{0}
\end{tabular}\endgroup%
}}\right.$}%
\begingroup \smaller\smaller\smaller\begin{tabular}{@{}c@{}}%
-1/13\\\phantom{0}\\\phantom{0}
\end{tabular}\endgroup%
\kern3pt%
\begingroup \smaller\smaller\smaller\begin{tabular}{@{}c@{}}%
\phantom{0}\\3/13\\\phantom{0}
\end{tabular}\endgroup%
\kern3pt%
\begingroup \smaller\smaller\smaller\begin{tabular}{@{}c@{}}%
\phantom{0}\\\phantom{0}\\5
\end{tabular}\endgroup%
{$\left.\llap{\phantom{%
\begingroup \smaller\smaller\smaller\begin{tabular}{@{}c@{}}%
\phantom{0}\\\phantom{0}\\\phantom{0}
\end{tabular}\endgroup%
}}\!\right]$}%
{$\left[\!\llap{\phantom{%
\begingroup \smaller\smaller\smaller\begin{tabular}{@{}c@{}}%
0\\0\\0
\end{tabular}\endgroup%
}}\right.$}%
\begingroup \smaller\smaller\smaller\begin{tabular}{@{}c@{}}%
12\\10\\0
\end{tabular}\endgroup%
\kern3pt%
\begingroup \smaller\smaller\smaller\begin{tabular}{@{}c@{}}%
4\\-1\\1
\end{tabular}\endgroup%
\kern3pt%
\begingroup \smaller\smaller\smaller\begin{tabular}{@{}c@{}}%
3\\-4\\0
\end{tabular}\endgroup%
{$\left.\llap{\phantom{%
\begingroup \smaller\smaller\smaller\begin{tabular}{@{}c@{}}%
0\\0\\0
\end{tabular}\endgroup%
}}\!\right]$}%
}%
\ifdim\wd\matricesbox>\halfwidth\myboxwidth=\hsize\else\myboxwidth=\halfwidth\fi
\vbox{%
\ifdim\myboxwidth=\hsize
\setbox\onelinebox=\hbox{%
\vbox{\hbox{%
$\Pi_{4,14}$ spans $L_{31.3}$%
}\hbox{%
$6|62|2\rtimes D_{2}$%
}%
}%
\hfill\copy\matricesbox
}%
\ifdim\wd\onelinebox>\myboxwidth
\hbox to \myboxwidth{%
$\Pi_{4,14}$ spans $L_{31.3}$%
\hfil
$6|62|2\rtimes D_{2}$%
}%
\box\matricesbox
\else
\hbox to \myboxwidth{%
\unhbox\onelinebox
}%
\fi
\else
\hbox to \myboxwidth{%
$\Pi_{4,14}$ spans $L_{31.3}$%
\hfil}%
\hbox to \myboxwidth{%
$6|62|2\rtimes D_{2}$%
\hfil}%
\box\matricesbox
\fi
}%
\hfill\discretionary{}{}{}%
\setbox\matricesbox=\hbox{%
{$\left[\!\llap{\phantom{%
\begingroup \smaller\smaller\smaller\begin{tabular}{@{}c@{}}%
\phantom{0}\\\phantom{0}\\\phantom{0}
\end{tabular}\endgroup%
}}\right.$}%
\begingroup \smaller\smaller\smaller\begin{tabular}{@{}c@{}}%
-1/7\\\phantom{0}\\\phantom{0}
\end{tabular}\endgroup%
\kern3pt%
\begingroup \smaller\smaller\smaller\begin{tabular}{@{}c@{}}%
\phantom{0}\\9/14\\\phantom{0}
\end{tabular}\endgroup%
\kern3pt%
\begingroup \smaller\smaller\smaller\begin{tabular}{@{}c@{}}%
\phantom{0}\\\phantom{0}\\21/2
\end{tabular}\endgroup%
{$\left.\llap{\phantom{%
\begingroup \smaller\smaller\smaller\begin{tabular}{@{}c@{}}%
\phantom{0}\\\phantom{0}\\\phantom{0}
\end{tabular}\endgroup%
}}\!\right]$}%
{$\left[\!\llap{\phantom{%
\begingroup \smaller\smaller\smaller\begin{tabular}{@{}c@{}}%
0\\0\\0
\end{tabular}\endgroup%
}}\right.$}%
\begingroup \smaller\smaller\smaller\begin{tabular}{@{}c@{}}%
18\\10\\0
\end{tabular}\endgroup%
\kern3pt%
\begingroup \smaller\smaller\smaller\begin{tabular}{@{}c@{}}%
6\\1\\-1
\end{tabular}\endgroup%
\kern3pt%
\begingroup \smaller\smaller\smaller\begin{tabular}{@{}c@{}}%
2\\-2\\0
\end{tabular}\endgroup%
{$\left.\llap{\phantom{%
\begingroup \smaller\smaller\smaller\begin{tabular}{@{}c@{}}%
0\\0\\0
\end{tabular}\endgroup%
}}\!\right]$}%
}%
\ifdim\wd\matricesbox>\halfwidth\myboxwidth=\hsize\else\myboxwidth=\halfwidth\fi
\vbox{%
\ifdim\myboxwidth=\hsize
\setbox\onelinebox=\hbox{%
\vbox{\hbox{%
$\Pi_{4,15}$ spans $L_{228.1}$%
}\hbox{%
$6|66|6\rtimes D_{2}$%
}%
}%
\hfill\copy\matricesbox
}%
\ifdim\wd\onelinebox>\myboxwidth
\hbox to \myboxwidth{%
$\Pi_{4,15}$ spans $L_{228.1}$%
\hfil
$6|66|6\rtimes D_{2}$%
}%
\box\matricesbox
\else
\hbox to \myboxwidth{%
\unhbox\onelinebox
}%
\fi
\else
\hbox to \myboxwidth{%
$\Pi_{4,15}$ spans $L_{228.1}$%
\hfil}%
\hbox to \myboxwidth{%
$6|66|6\rtimes D_{2}$%
\hfil}%
\box\matricesbox
\fi
}%
\hfill\discretionary{}{}{}%
\setbox\matricesbox=\hbox{%
{$\left[\!\llap{\phantom{%
\begingroup \smaller\smaller\smaller\begin{tabular}{@{}c@{}}%
\phantom{0}\\\phantom{0}\\\phantom{0}
\end{tabular}\endgroup%
}}\right.$}%
\begingroup \smaller\smaller\smaller\begin{tabular}{@{}c@{}}%
-1/32\\\phantom{0}\\\phantom{0}
\end{tabular}\endgroup%
\kern3pt%
\begingroup \smaller\smaller\smaller\begin{tabular}{@{}c@{}}%
\phantom{0}\\5/8\\\phantom{0}
\end{tabular}\endgroup%
\kern3pt%
\begingroup \smaller\smaller\smaller\begin{tabular}{@{}c@{}}%
\phantom{0}\\\phantom{0}\\25/2
\end{tabular}\endgroup%
{$\left.\llap{\phantom{%
\begingroup \smaller\smaller\smaller\begin{tabular}{@{}c@{}}%
\phantom{0}\\\phantom{0}\\\phantom{0}
\end{tabular}\endgroup%
}}\!\right]$}%
{$\left[\!\llap{\phantom{%
\begingroup \smaller\smaller\smaller\begin{tabular}{@{}c@{}}%
0\\0\\0
\end{tabular}\endgroup%
}}\right.$}%
\begingroup \smaller\smaller\smaller\begin{tabular}{@{}c@{}}%
40\\-12\\0
\end{tabular}\endgroup%
\kern3pt%
\begingroup \smaller\smaller\smaller\begin{tabular}{@{}c@{}}%
10\\1\\-1
\end{tabular}\endgroup%
\kern3pt%
\begingroup \smaller\smaller\smaller\begin{tabular}{@{}c@{}}%
8\\4\\0
\end{tabular}\endgroup%
{$\left.\llap{\phantom{%
\begingroup \smaller\smaller\smaller\begin{tabular}{@{}c@{}}%
0\\0\\0
\end{tabular}\endgroup%
}}\!\right]$}%
}%
\ifdim\wd\matricesbox>\halfwidth\myboxwidth=\hsize\else\myboxwidth=\halfwidth\fi
\vbox{%
\ifdim\myboxwidth=\hsize
\setbox\onelinebox=\hbox{%
\vbox{\hbox{%
$\Pi_{4,16}$ spans $L_{10.1}$%
}\hbox{%
$\infty|\infty2|2\rtimes D_{2}$%
}%
}%
\hfill\copy\matricesbox
}%
\ifdim\wd\onelinebox>\myboxwidth
\hbox to \myboxwidth{%
$\Pi_{4,16}$ spans $L_{10.1}$%
\hfil
$\infty|\infty2|2\rtimes D_{2}$%
}%
\box\matricesbox
\else
\hbox to \myboxwidth{%
\unhbox\onelinebox
}%
\fi
\else
\hbox to \myboxwidth{%
$\Pi_{4,16}$ spans $L_{10.1}$%
\hfil}%
\hbox to \myboxwidth{%
$\infty|\infty2|2\rtimes D_{2}$%
\hfil}%
\box\matricesbox
\fi
}%
\hfill\discretionary{}{}{}%
\setbox\matricesbox=\hbox{%
{$\left[\!\llap{\phantom{%
\begingroup \smaller\smaller\smaller\begin{tabular}{@{}c@{}}%
\phantom{0}\\\phantom{0}\\\phantom{0}
\end{tabular}\endgroup%
}}\right.$}%
\begingroup \smaller\smaller\smaller\begin{tabular}{@{}c@{}}%
-1/2\\\phantom{0}\\\phantom{0}
\end{tabular}\endgroup%
\kern3pt%
\begingroup \smaller\smaller\smaller\begin{tabular}{@{}c@{}}%
\phantom{0}\\1/2\\\phantom{0}
\end{tabular}\endgroup%
\kern3pt%
\begingroup \smaller\smaller\smaller\begin{tabular}{@{}c@{}}%
\phantom{0}\\\phantom{0}\\1
\end{tabular}\endgroup%
{$\left.\llap{\phantom{%
\begingroup \smaller\smaller\smaller\begin{tabular}{@{}c@{}}%
\phantom{0}\\\phantom{0}\\\phantom{0}
\end{tabular}\endgroup%
}}\!\right]$}%
{$\left[\!\llap{\phantom{%
\begingroup \smaller\smaller\smaller\begin{tabular}{@{}c@{}}%
0\\0\\0
\end{tabular}\endgroup%
}}\right.$}%
\begingroup \smaller\smaller\smaller\begin{tabular}{@{}c@{}}%
4\\4\\2
\end{tabular}\endgroup%
\kern3pt%
\begingroup \smaller\smaller\smaller\begin{tabular}{@{}c@{}}%
1\\-1\\1
\end{tabular}\endgroup%
{$\left.\llap{\phantom{%
\begingroup \smaller\smaller\smaller\begin{tabular}{@{}c@{}}%
0\\0\\0
\end{tabular}\endgroup%
}}\!\right]$}%
}%
\ifdim\wd\matricesbox>\halfwidth\myboxwidth=\hsize\else\myboxwidth=\halfwidth\fi
\vbox{%
\ifdim\myboxwidth=\hsize
\setbox\onelinebox=\hbox{%
\vbox{\hbox{%
$\Pi_{4,17}=\hbox{GN}_{31}$ spans $L_{140.4}$%
}\hbox{%
$\infty\slashinfty\infty\slashinfty\rtimes D_{2}$ (shared)%
}%
}%
\hfill\copy\matricesbox
}%
\ifdim\wd\onelinebox>\myboxwidth
\hbox to \myboxwidth{%
$\Pi_{4,17}=\hbox{GN}_{31}$ spans $L_{140.4}$%
\hfil
$\infty\slashinfty\infty\slashinfty\rtimes D_{2}$ (shared)%
}%
\box\matricesbox
\else
\hbox to \myboxwidth{%
\unhbox\onelinebox
}%
\fi
\else
\hbox to \myboxwidth{%
$\Pi_{4,17}=\hbox{GN}_{31}$ spans $L_{140.4}$%
\hfil}%
\hbox to \myboxwidth{%
$\infty\slashinfty\infty\slashinfty\rtimes D_{2}$ (shared)%
\hfil}%
\box\matricesbox
\fi
}%
\hfill\discretionary{}{}{}%
\setbox\matricesbox=\hbox{%
{$\left[\!\llap{\phantom{%
\begingroup \smaller\smaller\smaller\begin{tabular}{@{}c@{}}%
\phantom{0}\\\phantom{0}\\\phantom{0}
\end{tabular}\endgroup%
}}\right.$}%
\begingroup \smaller\smaller\smaller\begin{tabular}{@{}c@{}}%
-1/5\\\phantom{0}\\\phantom{0}
\end{tabular}\endgroup%
\kern3pt%
\begingroup \smaller\smaller\smaller\begin{tabular}{@{}c@{}}%
\phantom{0}\\3/10\\\phantom{0}
\end{tabular}\endgroup%
\kern3pt%
\begingroup \smaller\smaller\smaller\begin{tabular}{@{}c@{}}%
\phantom{0}\\\phantom{0}\\5/2
\end{tabular}\endgroup%
{$\left.\llap{\phantom{%
\begingroup \smaller\smaller\smaller\begin{tabular}{@{}c@{}}%
\phantom{0}\\\phantom{0}\\\phantom{0}
\end{tabular}\endgroup%
}}\!\right]$}%
{$\left[\!\llap{\phantom{%
\begingroup \smaller\smaller\smaller\begin{tabular}{@{}c@{}}%
0\\0\\0
\end{tabular}\endgroup%
}}\right.$}%
\begingroup \smaller\smaller\smaller\begin{tabular}{@{}c@{}}%
1\\-2\\0
\end{tabular}\endgroup%
\kern3pt%
\begingroup \smaller\smaller\smaller\begin{tabular}{@{}c@{}}%
2\\1\\-1
\end{tabular}\endgroup%
\kern3pt%
\begingroup \smaller\smaller\smaller\begin{tabular}{@{}c@{}}%
3\\4\\0
\end{tabular}\endgroup%
{$\left.\llap{\phantom{%
\begingroup \smaller\smaller\smaller\begin{tabular}{@{}c@{}}%
0\\0\\0
\end{tabular}\endgroup%
}}\!\right]$}%
}%
\ifdim\wd\matricesbox>\halfwidth\myboxwidth=\hsize\else\myboxwidth=\halfwidth\fi
\vbox{%
\ifdim\myboxwidth=\hsize
\setbox\onelinebox=\hbox{%
\vbox{\hbox{%
$\Pi_{4,18}$ spans $L_{19.2}$%
}\hbox{%
$4|42|2\rtimes D_{2}$ (shared)%
}%
}%
\hfill\copy\matricesbox
}%
\ifdim\wd\onelinebox>\myboxwidth
\hbox to \myboxwidth{%
$\Pi_{4,18}$ spans $L_{19.2}$%
\hfil
$4|42|2\rtimes D_{2}$ (shared)%
}%
\box\matricesbox
\else
\hbox to \myboxwidth{%
\unhbox\onelinebox
}%
\fi
\else
\hbox to \myboxwidth{%
$\Pi_{4,18}$ spans $L_{19.2}$%
\hfil}%
\hbox to \myboxwidth{%
$4|42|2\rtimes D_{2}$ (shared)%
\hfil}%
\box\matricesbox
\fi
}%
\hfill\discretionary{}{}{}%
\setbox\matricesbox=\hbox{%
{$\left[\!\llap{\phantom{%
\begingroup \smaller\smaller\smaller\begin{tabular}{@{}c@{}}%
\phantom{0}\\\phantom{0}\\\phantom{0}
\end{tabular}\endgroup%
}}\right.$}%
\begingroup \smaller\smaller\smaller\begin{tabular}{@{}c@{}}%
-1/7\\\phantom{0}\\\phantom{0}
\end{tabular}\endgroup%
\kern3pt%
\begingroup \smaller\smaller\smaller\begin{tabular}{@{}c@{}}%
\phantom{0}\\2/7\\\phantom{0}
\end{tabular}\endgroup%
\kern3pt%
\begingroup \smaller\smaller\smaller\begin{tabular}{@{}c@{}}%
\phantom{0}\\\phantom{0}\\6
\end{tabular}\endgroup%
{$\left.\llap{\phantom{%
\begingroup \smaller\smaller\smaller\begin{tabular}{@{}c@{}}%
\phantom{0}\\\phantom{0}\\\phantom{0}
\end{tabular}\endgroup%
}}\!\right]$}%
{$\left[\!\llap{\phantom{%
\begingroup \smaller\smaller\smaller\begin{tabular}{@{}c@{}}%
0\\0\\0
\end{tabular}\endgroup%
}}\right.$}%
\begingroup \smaller\smaller\smaller\begin{tabular}{@{}c@{}}%
2\\-3\\0
\end{tabular}\endgroup%
\kern3pt%
\begingroup \smaller\smaller\smaller\begin{tabular}{@{}c@{}}%
4\\1\\1
\end{tabular}\endgroup%
\kern3pt%
\begingroup \smaller\smaller\smaller\begin{tabular}{@{}c@{}}%
1\\2\\0
\end{tabular}\endgroup%
{$\left.\llap{\phantom{%
\begingroup \smaller\smaller\smaller\begin{tabular}{@{}c@{}}%
0\\0\\0
\end{tabular}\endgroup%
}}\!\right]$}%
}%
\ifdim\wd\matricesbox>\halfwidth\myboxwidth=\hsize\else\myboxwidth=\halfwidth\fi
\vbox{%
\ifdim\myboxwidth=\hsize
\setbox\onelinebox=\hbox{%
\vbox{\hbox{%
$\Pi_{4,19}$ spans $L_{123.5}$%
}\hbox{%
$4|42|2\rtimes D_{2}$ (shared)%
}%
}%
\hfill\copy\matricesbox
}%
\ifdim\wd\onelinebox>\myboxwidth
\hbox to \myboxwidth{%
$\Pi_{4,19}$ spans $L_{123.5}$%
\hfil
$4|42|2\rtimes D_{2}$ (shared)%
}%
\box\matricesbox
\else
\hbox to \myboxwidth{%
\unhbox\onelinebox
}%
\fi
\else
\hbox to \myboxwidth{%
$\Pi_{4,19}$ spans $L_{123.5}$%
\hfil}%
\hbox to \myboxwidth{%
$4|42|2\rtimes D_{2}$ (shared)%
\hfil}%
\box\matricesbox
\fi
}%
\hfill\discretionary{}{}{}%
\setbox\matricesbox=\hbox{%
{$\left[\!\llap{\phantom{%
\begingroup \smaller\smaller\smaller\begin{tabular}{@{}c@{}}%
\phantom{0}\\\phantom{0}\\\phantom{0}
\end{tabular}\endgroup%
}}\right.$}%
\begingroup \smaller\smaller\smaller\begin{tabular}{@{}c@{}}%
-4/9\\\phantom{0}\\\phantom{0}
\end{tabular}\endgroup%
\kern3pt%
\begingroup \smaller\smaller\smaller\begin{tabular}{@{}c@{}}%
\phantom{0}\\1/9\\\phantom{0}
\end{tabular}\endgroup%
\kern3pt%
\begingroup \smaller\smaller\smaller\begin{tabular}{@{}c@{}}%
\phantom{0}\\\phantom{0}\\1
\end{tabular}\endgroup%
{$\left.\llap{\phantom{%
\begingroup \smaller\smaller\smaller\begin{tabular}{@{}c@{}}%
\phantom{0}\\\phantom{0}\\\phantom{0}
\end{tabular}\endgroup%
}}\!\right]$}%
{$\left[\!\llap{\phantom{%
\begingroup \smaller\smaller\smaller\begin{tabular}{@{}c@{}}%
0\\0\\0
\end{tabular}\endgroup%
}}\right.$}%
\begingroup \smaller\smaller\smaller\begin{tabular}{@{}c@{}}%
1\\2\\1
\end{tabular}\endgroup%
\kern3pt%
\begingroup \smaller\smaller\smaller\begin{tabular}{@{}c@{}}%
2\\-5\\1
\end{tabular}\endgroup%
{$\left.\llap{\phantom{%
\begingroup \smaller\smaller\smaller\begin{tabular}{@{}c@{}}%
0\\0\\0
\end{tabular}\endgroup%
}}\!\right]$}%
}%
\ifdim\wd\matricesbox>\halfwidth\myboxwidth=\hsize\else\myboxwidth=\halfwidth\fi
\vbox{%
\ifdim\myboxwidth=\hsize
\setbox\onelinebox=\hbox{%
\vbox{\hbox{%
$\Pi_{4,20}$ spans $L_{145.1}$%
}\hbox{%
$4\slashinfty4\slashtwo\rtimes D_{2}$ (shared)%
}%
}%
\hfill\copy\matricesbox
}%
\ifdim\wd\onelinebox>\myboxwidth
\hbox to \myboxwidth{%
$\Pi_{4,20}$ spans $L_{145.1}$%
\hfil
$4\slashinfty4\slashtwo\rtimes D_{2}$ (shared)%
}%
\box\matricesbox
\else
\hbox to \myboxwidth{%
\unhbox\onelinebox
}%
\fi
\else
\hbox to \myboxwidth{%
$\Pi_{4,20}$ spans $L_{145.1}$%
\hfil}%
\hbox to \myboxwidth{%
$4\slashinfty4\slashtwo\rtimes D_{2}$ (shared)%
\hfil}%
\box\matricesbox
\fi
}%
\hfill\discretionary{}{}{}%
\setbox\matricesbox=\hbox{%
{$\left[\!\llap{\phantom{%
\begingroup \smaller\smaller\smaller\begin{tabular}{@{}c@{}}%
\phantom{0}\\\phantom{0}\\\phantom{0}
\end{tabular}\endgroup%
}}\right.$}%
\begingroup \smaller\smaller\smaller\begin{tabular}{@{}c@{}}%
-1/5\\\phantom{0}\\\phantom{0}
\end{tabular}\endgroup%
\kern3pt%
\begingroup \smaller\smaller\smaller\begin{tabular}{@{}c@{}}%
\phantom{0}\\3/10\\\phantom{0}
\end{tabular}\endgroup%
\kern3pt%
\begingroup \smaller\smaller\smaller\begin{tabular}{@{}c@{}}%
\phantom{0}\\\phantom{0}\\9/2
\end{tabular}\endgroup%
{$\left.\llap{\phantom{%
\begingroup \smaller\smaller\smaller\begin{tabular}{@{}c@{}}%
\phantom{0}\\\phantom{0}\\\phantom{0}
\end{tabular}\endgroup%
}}\!\right]$}%
{$\left[\!\llap{\phantom{%
\begingroup \smaller\smaller\smaller\begin{tabular}{@{}c@{}}%
0\\0\\0
\end{tabular}\endgroup%
}}\right.$}%
\begingroup \smaller\smaller\smaller\begin{tabular}{@{}c@{}}%
3\\4\\0
\end{tabular}\endgroup%
\kern3pt%
\begingroup \smaller\smaller\smaller\begin{tabular}{@{}c@{}}%
3\\-1\\1
\end{tabular}\endgroup%
\kern3pt%
\begingroup \smaller\smaller\smaller\begin{tabular}{@{}c@{}}%
1\\-2\\0
\end{tabular}\endgroup%
{$\left.\llap{\phantom{%
\begingroup \smaller\smaller\smaller\begin{tabular}{@{}c@{}}%
0\\0\\0
\end{tabular}\endgroup%
}}\!\right]$}%
}%
\ifdim\wd\matricesbox>\halfwidth\myboxwidth=\hsize\else\myboxwidth=\halfwidth\fi
\vbox{%
\ifdim\myboxwidth=\hsize
\setbox\onelinebox=\hbox{%
\vbox{\hbox{%
$\Pi_{4,21}$ spans $L_{4.10}$%
}\hbox{%
$\infty|\infty2|2\rtimes D_{2}$ (shared)%
}%
}%
\hfill\copy\matricesbox
}%
\ifdim\wd\onelinebox>\myboxwidth
\hbox to \myboxwidth{%
$\Pi_{4,21}$ spans $L_{4.10}$%
\hfil
$\infty|\infty2|2\rtimes D_{2}$ (shared)%
}%
\box\matricesbox
\else
\hbox to \myboxwidth{%
\unhbox\onelinebox
}%
\fi
\else
\hbox to \myboxwidth{%
$\Pi_{4,21}$ spans $L_{4.10}$%
\hfil}%
\hbox to \myboxwidth{%
$\infty|\infty2|2\rtimes D_{2}$ (shared)%
\hfil}%
\box\matricesbox
\fi
}%
\hfill\discretionary{}{}{}%
\setbox\matricesbox=\hbox{%
{$\left[\!\llap{\phantom{%
\begingroup \smaller\smaller\smaller\begin{tabular}{@{}c@{}}%
\phantom{0}\\\phantom{0}\\\phantom{0}
\end{tabular}\endgroup%
}}\right.$}%
\begingroup \smaller\smaller\smaller\begin{tabular}{@{}c@{}}%
-1/7\\\phantom{0}\\\phantom{0}
\end{tabular}\endgroup%
\kern3pt%
\begingroup \smaller\smaller\smaller\begin{tabular}{@{}c@{}}%
\phantom{0}\\1/14\\\phantom{0}
\end{tabular}\endgroup%
\kern3pt%
\begingroup \smaller\smaller\smaller\begin{tabular}{@{}c@{}}%
\phantom{0}\\\phantom{0}\\21/2
\end{tabular}\endgroup%
{$\left.\llap{\phantom{%
\begingroup \smaller\smaller\smaller\begin{tabular}{@{}c@{}}%
\phantom{0}\\\phantom{0}\\\phantom{0}
\end{tabular}\endgroup%
}}\!\right]$}%
{$\left[\!\llap{\phantom{%
\begingroup \smaller\smaller\smaller\begin{tabular}{@{}c@{}}%
0\\0\\0
\end{tabular}\endgroup%
}}\right.$}%
\begingroup \smaller\smaller\smaller\begin{tabular}{@{}c@{}}%
2\\6\\0
\end{tabular}\endgroup%
\kern3pt%
\begingroup \smaller\smaller\smaller\begin{tabular}{@{}c@{}}%
6\\-3\\1
\end{tabular}\endgroup%
\kern3pt%
\begingroup \smaller\smaller\smaller\begin{tabular}{@{}c@{}}%
1\\-4\\0
\end{tabular}\endgroup%
{$\left.\llap{\phantom{%
\begingroup \smaller\smaller\smaller\begin{tabular}{@{}c@{}}%
0\\0\\0
\end{tabular}\endgroup%
}}\!\right]$}%
}%
\ifdim\wd\matricesbox>\halfwidth\myboxwidth=\hsize\else\myboxwidth=\halfwidth\fi
\vbox{%
\ifdim\myboxwidth=\hsize
\setbox\onelinebox=\hbox{%
\vbox{\hbox{%
$\Pi_{4,22}$ spans $L_{22.2}$%
}\hbox{%
$6|62|2\rtimes D_{2}$%
}%
}%
\hfill\copy\matricesbox
}%
\ifdim\wd\onelinebox>\myboxwidth
\hbox to \myboxwidth{%
$\Pi_{4,22}$ spans $L_{22.2}$%
\hfil
$6|62|2\rtimes D_{2}$%
}%
\box\matricesbox
\else
\hbox to \myboxwidth{%
\unhbox\onelinebox
}%
\fi
\else
\hbox to \myboxwidth{%
$\Pi_{4,22}$ spans $L_{22.2}$%
\hfil}%
\hbox to \myboxwidth{%
$6|62|2\rtimes D_{2}$%
\hfil}%
\box\matricesbox
\fi
}%
\hfill\discretionary{}{}{}%
\setbox\matricesbox=\hbox{%
{$\left[\!\llap{\phantom{%
\begingroup \smaller\smaller\smaller\begin{tabular}{@{}c@{}}%
\phantom{0}\\\phantom{0}\\\phantom{0}
\end{tabular}\endgroup%
}}\right.$}%
\begingroup \smaller\smaller\smaller\begin{tabular}{@{}c@{}}%
-1/40\\\phantom{0}\\\phantom{0}
\end{tabular}\endgroup%
\kern3pt%
\begingroup \smaller\smaller\smaller\begin{tabular}{@{}c@{}}%
\phantom{0}\\21/10\\\phantom{0}
\end{tabular}\endgroup%
\kern3pt%
\begingroup \smaller\smaller\smaller\begin{tabular}{@{}c@{}}%
\phantom{0}\\\phantom{0}\\84
\end{tabular}\endgroup%
{$\left.\llap{\phantom{%
\begingroup \smaller\smaller\smaller\begin{tabular}{@{}c@{}}%
\phantom{0}\\\phantom{0}\\\phantom{0}
\end{tabular}\endgroup%
}}\!\right]$}%
{$\left[\!\llap{\phantom{%
\begingroup \smaller\smaller\smaller\begin{tabular}{@{}c@{}}%
0\\0\\0
\end{tabular}\endgroup%
}}\right.$}%
\begingroup \smaller\smaller\smaller\begin{tabular}{@{}c@{}}%
14\\-3\\0
\end{tabular}\endgroup%
\kern3pt%
\begingroup \smaller\smaller\smaller\begin{tabular}{@{}c@{}}%
42\\1\\1
\end{tabular}\endgroup%
\kern3pt%
\begingroup \smaller\smaller\smaller\begin{tabular}{@{}c@{}}%
2\\1\\0
\end{tabular}\endgroup%
{$\left.\llap{\phantom{%
\begingroup \smaller\smaller\smaller\begin{tabular}{@{}c@{}}%
0\\0\\0
\end{tabular}\endgroup%
}}\!\right]$}%
}%
\ifdim\wd\matricesbox>\halfwidth\myboxwidth=\hsize\else\myboxwidth=\halfwidth\fi
\vbox{%
\ifdim\myboxwidth=\hsize
\setbox\onelinebox=\hbox{%
\vbox{\hbox{%
$\Pi_{4,23}$ spans $L_{22.10}$%
}\hbox{%
$6|62|2\rtimes D_{2}$%
}%
}%
\hfill\copy\matricesbox
}%
\ifdim\wd\onelinebox>\myboxwidth
\hbox to \myboxwidth{%
$\Pi_{4,23}$ spans $L_{22.10}$%
\hfil
$6|62|2\rtimes D_{2}$%
}%
\box\matricesbox
\else
\hbox to \myboxwidth{%
\unhbox\onelinebox
}%
\fi
\else
\hbox to \myboxwidth{%
$\Pi_{4,23}$ spans $L_{22.10}$%
\hfil}%
\hbox to \myboxwidth{%
$6|62|2\rtimes D_{2}$%
\hfil}%
\box\matricesbox
\fi
}%
\hfill\discretionary{}{}{}%
\setbox\matricesbox=\hbox{%
{$\left[\!\llap{\phantom{%
\begingroup \smaller\smaller\smaller\begin{tabular}{@{}c@{}}%
\phantom{0}\\\phantom{0}\\\phantom{0}
\end{tabular}\endgroup%
}}\right.$}%
\begingroup \smaller\smaller\smaller\begin{tabular}{@{}c@{}}%
-1/5\\\phantom{0}\\\phantom{0}
\end{tabular}\endgroup%
\kern3pt%
\begingroup \smaller\smaller\smaller\begin{tabular}{@{}c@{}}%
\phantom{0}\\6/5\\\phantom{0}
\end{tabular}\endgroup%
\kern3pt%
\begingroup \smaller\smaller\smaller\begin{tabular}{@{}c@{}}%
\phantom{0}\\\phantom{0}\\6
\end{tabular}\endgroup%
{$\left.\llap{\phantom{%
\begingroup \smaller\smaller\smaller\begin{tabular}{@{}c@{}}%
\phantom{0}\\\phantom{0}\\\phantom{0}
\end{tabular}\endgroup%
}}\!\right]$}%
{$\left[\!\llap{\phantom{%
\begingroup \smaller\smaller\smaller\begin{tabular}{@{}c@{}}%
0\\0\\0
\end{tabular}\endgroup%
}}\right.$}%
\begingroup \smaller\smaller\smaller\begin{tabular}{@{}c@{}}%
1\\-1\\0
\end{tabular}\endgroup%
\kern3pt%
\begingroup \smaller\smaller\smaller\begin{tabular}{@{}c@{}}%
4\\1\\1
\end{tabular}\endgroup%
\kern3pt%
\begingroup \smaller\smaller\smaller\begin{tabular}{@{}c@{}}%
3\\2\\0
\end{tabular}\endgroup%
{$\left.\llap{\phantom{%
\begingroup \smaller\smaller\smaller\begin{tabular}{@{}c@{}}%
0\\0\\0
\end{tabular}\endgroup%
}}\!\right]$}%
}%
\ifdim\wd\matricesbox>\halfwidth\myboxwidth=\hsize\else\myboxwidth=\halfwidth\fi
\vbox{%
\ifdim\myboxwidth=\hsize
\setbox\onelinebox=\hbox{%
\vbox{\hbox{%
$\Pi_{4,24}$ spans $L_{4.18}$%
}\hbox{%
$\infty|\infty2|2\rtimes D_{2}$ (shared)%
}%
}%
\hfill\copy\matricesbox
}%
\ifdim\wd\onelinebox>\myboxwidth
\hbox to \myboxwidth{%
$\Pi_{4,24}$ spans $L_{4.18}$%
\hfil
$\infty|\infty2|2\rtimes D_{2}$ (shared)%
}%
\box\matricesbox
\else
\hbox to \myboxwidth{%
\unhbox\onelinebox
}%
\fi
\else
\hbox to \myboxwidth{%
$\Pi_{4,24}$ spans $L_{4.18}$%
\hfil}%
\hbox to \myboxwidth{%
$\infty|\infty2|2\rtimes D_{2}$ (shared)%
\hfil}%
\box\matricesbox
\fi
}%
\hfill\discretionary{}{}{}%
\setbox\matricesbox=\hbox{%
{$\left[\!\llap{\phantom{%
\begingroup \smaller\smaller\smaller\begin{tabular}{@{}c@{}}%
\phantom{0}\\\phantom{0}\\\phantom{0}
\end{tabular}\endgroup%
}}\right.$}%
\begingroup \smaller\smaller\smaller\begin{tabular}{@{}c@{}}%
-1/7\\\phantom{0}\\\phantom{0}
\end{tabular}\endgroup%
\kern3pt%
\begingroup \smaller\smaller\smaller\begin{tabular}{@{}c@{}}%
\phantom{0}\\2/7\\\phantom{0}
\end{tabular}\endgroup%
\kern3pt%
\begingroup \smaller\smaller\smaller\begin{tabular}{@{}c@{}}%
\phantom{0}\\\phantom{0}\\16
\end{tabular}\endgroup%
{$\left.\llap{\phantom{%
\begingroup \smaller\smaller\smaller\begin{tabular}{@{}c@{}}%
\phantom{0}\\\phantom{0}\\\phantom{0}
\end{tabular}\endgroup%
}}\!\right]$}%
{$\left[\!\llap{\phantom{%
\begingroup \smaller\smaller\smaller\begin{tabular}{@{}c@{}}%
0\\0\\0
\end{tabular}\endgroup%
}}\right.$}%
\begingroup \smaller\smaller\smaller\begin{tabular}{@{}c@{}}%
2\\-3\\0
\end{tabular}\endgroup%
\kern3pt%
\begingroup \smaller\smaller\smaller\begin{tabular}{@{}c@{}}%
8\\2\\1
\end{tabular}\endgroup%
\kern3pt%
\begingroup \smaller\smaller\smaller\begin{tabular}{@{}c@{}}%
1\\2\\0
\end{tabular}\endgroup%
{$\left.\llap{\phantom{%
\begingroup \smaller\smaller\smaller\begin{tabular}{@{}c@{}}%
0\\0\\0
\end{tabular}\endgroup%
}}\!\right]$}%
}%
\ifdim\wd\matricesbox>\halfwidth\myboxwidth=\hsize\else\myboxwidth=\halfwidth\fi
\vbox{%
\ifdim\myboxwidth=\hsize
\setbox\onelinebox=\hbox{%
\vbox{\hbox{%
$\Pi_{4,25}$ spans $L_{141.11}$%
}\hbox{%
$\infty|\infty2|2\rtimes D_{2}$ (shared)%
}%
}%
\hfill\copy\matricesbox
}%
\ifdim\wd\onelinebox>\myboxwidth
\hbox to \myboxwidth{%
$\Pi_{4,25}$ spans $L_{141.11}$%
\hfil
$\infty|\infty2|2\rtimes D_{2}$ (shared)%
}%
\box\matricesbox
\else
\hbox to \myboxwidth{%
\unhbox\onelinebox
}%
\fi
\else
\hbox to \myboxwidth{%
$\Pi_{4,25}$ spans $L_{141.11}$%
\hfil}%
\hbox to \myboxwidth{%
$\infty|\infty2|2\rtimes D_{2}$ (shared)%
\hfil}%
\box\matricesbox
\fi
}%
\hfill\discretionary{}{}{}%
\setbox\matricesbox=\hbox{%
{$\left[\!\llap{\phantom{%
\begingroup \smaller\smaller\smaller\begin{tabular}{@{}c@{}}%
\phantom{0}\\\phantom{0}\\\phantom{0}
\end{tabular}\endgroup%
}}\right.$}%
\begingroup \smaller\smaller\smaller\begin{tabular}{@{}c@{}}%
-1/40\\\phantom{0}\\\phantom{0}
\end{tabular}\endgroup%
\kern3pt%
\begingroup \smaller\smaller\smaller\begin{tabular}{@{}c@{}}%
\phantom{0}\\3/5\\\phantom{0}
\end{tabular}\endgroup%
\kern3pt%
\begingroup \smaller\smaller\smaller\begin{tabular}{@{}c@{}}%
\phantom{0}\\\phantom{0}\\3/2
\end{tabular}\endgroup%
{$\left.\llap{\phantom{%
\begingroup \smaller\smaller\smaller\begin{tabular}{@{}c@{}}%
\phantom{0}\\\phantom{0}\\\phantom{0}
\end{tabular}\endgroup%
}}\!\right]$}%
{$\left[\!\llap{\phantom{%
\begingroup \smaller\smaller\smaller\begin{tabular}{@{}c@{}}%
0\\0\\0
\end{tabular}\endgroup%
}}\right.$}%
\begingroup \smaller\smaller\smaller\begin{tabular}{@{}c@{}}%
2\\-1\\1
\end{tabular}\endgroup%
\kern3pt%
\begingroup \smaller\smaller\smaller\begin{tabular}{@{}c@{}}%
12\\4\\2
\end{tabular}\endgroup%
{$\left.\llap{\phantom{%
\begingroup \smaller\smaller\smaller\begin{tabular}{@{}c@{}}%
0\\0\\0
\end{tabular}\endgroup%
}}\!\right]$}%
}%
\ifdim\wd\matricesbox>\halfwidth\myboxwidth=\hsize\else\myboxwidth=\halfwidth\fi
\vbox{%
\ifdim\myboxwidth=\hsize
\setbox\onelinebox=\hbox{%
\vbox{\hbox{%
$\Pi_{4,26}$ spans $L_{3.4}$%
}\hbox{%
$\slashthree2\slashtwo2\rtimes D_{2}$ (shared)%
}%
}%
\hfill\copy\matricesbox
}%
\ifdim\wd\onelinebox>\myboxwidth
\hbox to \myboxwidth{%
$\Pi_{4,26}$ spans $L_{3.4}$%
\hfil
$\slashthree2\slashtwo2\rtimes D_{2}$ (shared)%
}%
\box\matricesbox
\else
\hbox to \myboxwidth{%
\unhbox\onelinebox
}%
\fi
\else
\hbox to \myboxwidth{%
$\Pi_{4,26}$ spans $L_{3.4}$%
\hfil}%
\hbox to \myboxwidth{%
$\slashthree2\slashtwo2\rtimes D_{2}$ (shared)%
\hfil}%
\box\matricesbox
\fi
}%
\hfill\discretionary{}{}{}%
\setbox\matricesbox=\hbox{%
{$\left[\!\llap{\phantom{%
\begingroup \smaller\smaller\smaller\begin{tabular}{@{}c@{}}%
\phantom{0}\\\phantom{0}\\\phantom{0}
\end{tabular}\endgroup%
}}\right.$}%
\begingroup \smaller\smaller\smaller\begin{tabular}{@{}c@{}}%
-1/28\\\phantom{0}\\\phantom{0}
\end{tabular}\endgroup%
\kern3pt%
\begingroup \smaller\smaller\smaller\begin{tabular}{@{}c@{}}%
\phantom{0}\\9/14\\\phantom{0}
\end{tabular}\endgroup%
\kern3pt%
\begingroup \smaller\smaller\smaller\begin{tabular}{@{}c@{}}%
\phantom{0}\\\phantom{0}\\3/2
\end{tabular}\endgroup%
{$\left.\llap{\phantom{%
\begingroup \smaller\smaller\smaller\begin{tabular}{@{}c@{}}%
\phantom{0}\\\phantom{0}\\\phantom{0}
\end{tabular}\endgroup%
}}\!\right]$}%
{$\left[\!\llap{\phantom{%
\begingroup \smaller\smaller\smaller\begin{tabular}{@{}c@{}}%
0\\0\\0
\end{tabular}\endgroup%
}}\right.$}%
\begingroup \smaller\smaller\smaller\begin{tabular}{@{}c@{}}%
2\\1\\1
\end{tabular}\endgroup%
\kern3pt%
\begingroup \smaller\smaller\smaller\begin{tabular}{@{}c@{}}%
18\\-5\\3
\end{tabular}\endgroup%
{$\left.\llap{\phantom{%
\begingroup \smaller\smaller\smaller\begin{tabular}{@{}c@{}}%
0\\0\\0
\end{tabular}\endgroup%
}}\!\right]$}%
}%
\ifdim\wd\matricesbox>\halfwidth\myboxwidth=\hsize\else\myboxwidth=\halfwidth\fi
\vbox{%
\ifdim\myboxwidth=\hsize
\setbox\onelinebox=\hbox{%
\vbox{\hbox{%
$\Pi_{4,27}=\hbox{GN}_{8}$ spans $L_{155.1}$%
}\hbox{%
$\slashthree2\slashthree2\rtimes D_{2}$ (shared)%
}%
}%
\hfill\copy\matricesbox
}%
\ifdim\wd\onelinebox>\myboxwidth
\hbox to \myboxwidth{%
$\Pi_{4,27}=\hbox{GN}_{8}$ spans $L_{155.1}$%
\hfil
$\slashthree2\slashthree2\rtimes D_{2}$ (shared)%
}%
\box\matricesbox
\else
\hbox to \myboxwidth{%
\unhbox\onelinebox
}%
\fi
\else
\hbox to \myboxwidth{%
$\Pi_{4,27}=\hbox{GN}_{8}$ spans $L_{155.1}$%
\hfil}%
\hbox to \myboxwidth{%
$\slashthree2\slashthree2\rtimes D_{2}$ (shared)%
\hfil}%
\box\matricesbox
\fi
}%
\hfill\discretionary{}{}{}%
\setbox\matricesbox=\hbox{%
{$\left[\!\llap{\phantom{%
\begingroup \smaller\smaller\smaller\begin{tabular}{@{}c@{}}%
\phantom{0}\\\phantom{0}\\\phantom{0}
\end{tabular}\endgroup%
}}\right.$}%
\begingroup \smaller\smaller\smaller\begin{tabular}{@{}c@{}}%
-1/24\\\phantom{0}\\\phantom{0}
\end{tabular}\endgroup%
\kern3pt%
\begingroup \smaller\smaller\smaller\begin{tabular}{@{}c@{}}%
\phantom{0}\\2/3\\\phantom{0}
\end{tabular}\endgroup%
\kern3pt%
\begingroup \smaller\smaller\smaller\begin{tabular}{@{}c@{}}%
\phantom{0}\\\phantom{0}\\3/2
\end{tabular}\endgroup%
{$\left.\llap{\phantom{%
\begingroup \smaller\smaller\smaller\begin{tabular}{@{}c@{}}%
\phantom{0}\\\phantom{0}\\\phantom{0}
\end{tabular}\endgroup%
}}\!\right]$}%
{$\left[\!\llap{\phantom{%
\begingroup \smaller\smaller\smaller\begin{tabular}{@{}c@{}}%
0\\0\\0
\end{tabular}\endgroup%
}}\right.$}%
\begingroup \smaller\smaller\smaller\begin{tabular}{@{}c@{}}%
2\\-1\\1
\end{tabular}\endgroup%
\kern3pt%
\begingroup \smaller\smaller\smaller\begin{tabular}{@{}c@{}}%
24\\6\\4
\end{tabular}\endgroup%
{$\left.\llap{\phantom{%
\begingroup \smaller\smaller\smaller\begin{tabular}{@{}c@{}}%
0\\0\\0
\end{tabular}\endgroup%
}}\!\right]$}%
}%
\ifdim\wd\matricesbox>\halfwidth\myboxwidth=\hsize\else\myboxwidth=\halfwidth\fi
\vbox{%
\ifdim\myboxwidth=\hsize
\setbox\onelinebox=\hbox{%
\vbox{\hbox{%
$\Pi_{4,28}$ spans $L_{7.13}$%
}\hbox{%
$\slashthree2\slashinfty2\rtimes D_{2}$ (shared)%
}%
}%
\hfill\copy\matricesbox
}%
\ifdim\wd\onelinebox>\myboxwidth
\hbox to \myboxwidth{%
$\Pi_{4,28}$ spans $L_{7.13}$%
\hfil
$\slashthree2\slashinfty2\rtimes D_{2}$ (shared)%
}%
\box\matricesbox
\else
\hbox to \myboxwidth{%
\unhbox\onelinebox
}%
\fi
\else
\hbox to \myboxwidth{%
$\Pi_{4,28}$ spans $L_{7.13}$%
\hfil}%
\hbox to \myboxwidth{%
$\slashthree2\slashinfty2\rtimes D_{2}$ (shared)%
\hfil}%
\box\matricesbox
\fi
}%
\hfill\discretionary{}{}{}%
\setbox\matricesbox=\hbox{%
{$\left[\!\llap{\phantom{%
\begingroup \smaller\smaller\smaller\begin{tabular}{@{}c@{}}%
\phantom{0}\\\phantom{0}\\\phantom{0}
\end{tabular}\endgroup%
}}\right.$}%
\begingroup \smaller\smaller\smaller\begin{tabular}{@{}c@{}}%
-1/19\\\phantom{0}\\\phantom{0}
\end{tabular}\endgroup%
\kern3pt%
\begingroup \smaller\smaller\smaller\begin{tabular}{@{}c@{}}%
\phantom{0}\\3/38\\\phantom{0}
\end{tabular}\endgroup%
\kern3pt%
\begingroup \smaller\smaller\smaller\begin{tabular}{@{}c@{}}%
\phantom{0}\\\phantom{0}\\3/2
\end{tabular}\endgroup%
{$\left.\llap{\phantom{%
\begingroup \smaller\smaller\smaller\begin{tabular}{@{}c@{}}%
\phantom{0}\\\phantom{0}\\\phantom{0}
\end{tabular}\endgroup%
}}\!\right]$}%
{$\left[\!\llap{\phantom{%
\begingroup \smaller\smaller\smaller\begin{tabular}{@{}c@{}}%
0\\0\\0
\end{tabular}\endgroup%
}}\right.$}%
\begingroup \smaller\smaller\smaller\begin{tabular}{@{}c@{}}%
2\\3\\-1
\end{tabular}\endgroup%
\kern3pt%
\begingroup \smaller\smaller\smaller\begin{tabular}{@{}c@{}}%
3\\-5\\-1
\end{tabular}\endgroup%
{$\left.\llap{\phantom{%
\begingroup \smaller\smaller\smaller\begin{tabular}{@{}c@{}}%
0\\0\\0
\end{tabular}\endgroup%
}}\!\right]$}%
}%
\ifdim\wd\matricesbox>\halfwidth\myboxwidth=\hsize\else\myboxwidth=\halfwidth\fi
\vbox{%
\ifdim\myboxwidth=\hsize
\setbox\onelinebox=\hbox{%
\vbox{\hbox{%
$\Pi_{4,29}$ spans $L_{3.3}$%
}\hbox{%
$\slashthree2\slashtwo2\rtimes D_{2}$ (shared)%
}%
}%
\hfill\copy\matricesbox
}%
\ifdim\wd\onelinebox>\myboxwidth
\hbox to \myboxwidth{%
$\Pi_{4,29}$ spans $L_{3.3}$%
\hfil
$\slashthree2\slashtwo2\rtimes D_{2}$ (shared)%
}%
\box\matricesbox
\else
\hbox to \myboxwidth{%
\unhbox\onelinebox
}%
\fi
\else
\hbox to \myboxwidth{%
$\Pi_{4,29}$ spans $L_{3.3}$%
\hfil}%
\hbox to \myboxwidth{%
$\slashthree2\slashtwo2\rtimes D_{2}$ (shared)%
\hfil}%
\box\matricesbox
\fi
}%
\hfill\discretionary{}{}{}%
\setbox\matricesbox=\hbox{%
{$\left[\!\llap{\phantom{%
\begingroup \smaller\smaller\smaller\begin{tabular}{@{}c@{}}%
\phantom{0}\\\phantom{0}\\\phantom{0}
\end{tabular}\endgroup%
}}\right.$}%
\begingroup \smaller\smaller\smaller\begin{tabular}{@{}c@{}}%
-1/10\\\phantom{0}\\\phantom{0}
\end{tabular}\endgroup%
\kern3pt%
\begingroup \smaller\smaller\smaller\begin{tabular}{@{}c@{}}%
\phantom{0}\\9/10\\\phantom{0}
\end{tabular}\endgroup%
\kern3pt%
\begingroup \smaller\smaller\smaller\begin{tabular}{@{}c@{}}%
\phantom{0}\\\phantom{0}\\3/2
\end{tabular}\endgroup%
{$\left.\llap{\phantom{%
\begingroup \smaller\smaller\smaller\begin{tabular}{@{}c@{}}%
\phantom{0}\\\phantom{0}\\\phantom{0}
\end{tabular}\endgroup%
}}\!\right]$}%
{$\left[\!\llap{\phantom{%
\begingroup \smaller\smaller\smaller\begin{tabular}{@{}c@{}}%
0\\0\\0
\end{tabular}\endgroup%
}}\right.$}%
\begingroup \smaller\smaller\smaller\begin{tabular}{@{}c@{}}%
2\\1\\1
\end{tabular}\endgroup%
\kern3pt%
\begingroup \smaller\smaller\smaller\begin{tabular}{@{}c@{}}%
6\\-2\\2
\end{tabular}\endgroup%
{$\left.\llap{\phantom{%
\begingroup \smaller\smaller\smaller\begin{tabular}{@{}c@{}}%
0\\0\\0
\end{tabular}\endgroup%
}}\!\right]$}%
}%
\ifdim\wd\matricesbox>\halfwidth\myboxwidth=\hsize\else\myboxwidth=\halfwidth\fi
\vbox{%
\ifdim\myboxwidth=\hsize
\setbox\onelinebox=\hbox{%
\vbox{\hbox{%
$\Pi_{4,30}$ spans $L_{7.8}$%
}\hbox{%
$\slashthree2\slashinfty2\rtimes D_{2}$ (shared)%
}%
}%
\hfill\copy\matricesbox
}%
\ifdim\wd\onelinebox>\myboxwidth
\hbox to \myboxwidth{%
$\Pi_{4,30}$ spans $L_{7.8}$%
\hfil
$\slashthree2\slashinfty2\rtimes D_{2}$ (shared)%
}%
\box\matricesbox
\else
\hbox to \myboxwidth{%
\unhbox\onelinebox
}%
\fi
\else
\hbox to \myboxwidth{%
$\Pi_{4,30}$ spans $L_{7.8}$%
\hfil}%
\hbox to \myboxwidth{%
$\slashthree2\slashinfty2\rtimes D_{2}$ (shared)%
\hfil}%
\box\matricesbox
\fi
}%
\hfill\discretionary{}{}{}%
\setbox\matricesbox=\hbox{%
{$\left[\!\llap{\phantom{%
\begingroup \smaller\smaller\smaller\begin{tabular}{@{}c@{}}%
\phantom{0}\\\phantom{0}\\\phantom{0}
\end{tabular}\endgroup%
}}\right.$}%
\begingroup \smaller\smaller\smaller\begin{tabular}{@{}c@{}}%
-1/40\\\phantom{0}\\\phantom{0}
\end{tabular}\endgroup%
\kern3pt%
\begingroup \smaller\smaller\smaller\begin{tabular}{@{}c@{}}%
\phantom{0}\\12/5\\\phantom{0}
\end{tabular}\endgroup%
\kern3pt%
\begingroup \smaller\smaller\smaller\begin{tabular}{@{}c@{}}%
\phantom{0}\\\phantom{0}\\1/2
\end{tabular}\endgroup%
{$\left.\llap{\phantom{%
\begingroup \smaller\smaller\smaller\begin{tabular}{@{}c@{}}%
\phantom{0}\\\phantom{0}\\\phantom{0}
\end{tabular}\endgroup%
}}\!\right]$}%
{$\left[\!\llap{\phantom{%
\begingroup \smaller\smaller\smaller\begin{tabular}{@{}c@{}}%
0\\0\\0
\end{tabular}\endgroup%
}}\right.$}%
\begingroup \smaller\smaller\smaller\begin{tabular}{@{}c@{}}%
6\\-1\\3
\end{tabular}\endgroup%
\kern3pt%
\begingroup \smaller\smaller\smaller\begin{tabular}{@{}c@{}}%
4\\1\\2
\end{tabular}\endgroup%
{$\left.\llap{\phantom{%
\begingroup \smaller\smaller\smaller\begin{tabular}{@{}c@{}}%
0\\0\\0
\end{tabular}\endgroup%
}}\!\right]$}%
}%
\ifdim\wd\matricesbox>\halfwidth\myboxwidth=\hsize\else\myboxwidth=\halfwidth\fi
\vbox{%
\ifdim\myboxwidth=\hsize
\setbox\onelinebox=\hbox{%
\vbox{\hbox{%
$\Pi_{4,31}$ spans $L_{3.1}$%
}\hbox{%
$\slashthree2\slashtwo2\rtimes D_{2}$ (shared)%
}%
}%
\hfill\copy\matricesbox
}%
\ifdim\wd\onelinebox>\myboxwidth
\hbox to \myboxwidth{%
$\Pi_{4,31}$ spans $L_{3.1}$%
\hfil
$\slashthree2\slashtwo2\rtimes D_{2}$ (shared)%
}%
\box\matricesbox
\else
\hbox to \myboxwidth{%
\unhbox\onelinebox
}%
\fi
\else
\hbox to \myboxwidth{%
$\Pi_{4,31}$ spans $L_{3.1}$%
\hfil}%
\hbox to \myboxwidth{%
$\slashthree2\slashtwo2\rtimes D_{2}$ (shared)%
\hfil}%
\box\matricesbox
\fi
}%
\hfill\discretionary{}{}{}%
\setbox\matricesbox=\hbox{%
{$\left[\!\llap{\phantom{%
\begingroup \smaller\smaller\smaller\begin{tabular}{@{}c@{}}%
\phantom{0}\\\phantom{0}\\\phantom{0}
\end{tabular}\endgroup%
}}\right.$}%
\begingroup \smaller\smaller\smaller\begin{tabular}{@{}c@{}}%
-3/56\\\phantom{0}\\\phantom{0}
\end{tabular}\endgroup%
\kern3pt%
\begingroup \smaller\smaller\smaller\begin{tabular}{@{}c@{}}%
\phantom{0}\\24/7\\\phantom{0}
\end{tabular}\endgroup%
\kern3pt%
\begingroup \smaller\smaller\smaller\begin{tabular}{@{}c@{}}%
\phantom{0}\\\phantom{0}\\1/2
\end{tabular}\endgroup%
{$\left.\llap{\phantom{%
\begingroup \smaller\smaller\smaller\begin{tabular}{@{}c@{}}%
\phantom{0}\\\phantom{0}\\\phantom{0}
\end{tabular}\endgroup%
}}\!\right]$}%
{$\left[\!\llap{\phantom{%
\begingroup \smaller\smaller\smaller\begin{tabular}{@{}c@{}}%
0\\0\\0
\end{tabular}\endgroup%
}}\right.$}%
\begingroup \smaller\smaller\smaller\begin{tabular}{@{}c@{}}%
6\\-1\\-3
\end{tabular}\endgroup%
\kern3pt%
\begingroup \smaller\smaller\smaller\begin{tabular}{@{}c@{}}%
8\\1\\-4
\end{tabular}\endgroup%
{$\left.\llap{\phantom{%
\begingroup \smaller\smaller\smaller\begin{tabular}{@{}c@{}}%
0\\0\\0
\end{tabular}\endgroup%
}}\!\right]$}%
}%
\ifdim\wd\matricesbox>\halfwidth\myboxwidth=\hsize\else\myboxwidth=\halfwidth\fi
\vbox{%
\ifdim\myboxwidth=\hsize
\setbox\onelinebox=\hbox{%
\vbox{\hbox{%
$\Pi_{4,32}$ spans $L_{7.9}$%
}\hbox{%
$\slashthree2\slashinfty2\rtimes D_{2}$ (shared)%
}%
}%
\hfill\copy\matricesbox
}%
\ifdim\wd\onelinebox>\myboxwidth
\hbox to \myboxwidth{%
$\Pi_{4,32}$ spans $L_{7.9}$%
\hfil
$\slashthree2\slashinfty2\rtimes D_{2}$ (shared)%
}%
\box\matricesbox
\else
\hbox to \myboxwidth{%
\unhbox\onelinebox
}%
\fi
\else
\hbox to \myboxwidth{%
$\Pi_{4,32}$ spans $L_{7.9}$%
\hfil}%
\hbox to \myboxwidth{%
$\slashthree2\slashinfty2\rtimes D_{2}$ (shared)%
\hfil}%
\box\matricesbox
\fi
}%
\hfill\discretionary{}{}{}%
\setbox\matricesbox=\hbox{%
{$\left[\!\llap{\phantom{%
\begingroup \smaller\smaller\smaller\begin{tabular}{@{}c@{}}%
\phantom{0}\\\phantom{0}\\\phantom{0}
\end{tabular}\endgroup%
}}\right.$}%
\begingroup \smaller\smaller\smaller\begin{tabular}{@{}c@{}}%
-1/9\\\phantom{0}\\\phantom{0}
\end{tabular}\endgroup%
\kern3pt%
\begingroup \smaller\smaller\smaller\begin{tabular}{@{}c@{}}%
\phantom{0}\\1/9\\\phantom{0}
\end{tabular}\endgroup%
\kern3pt%
\begingroup \smaller\smaller\smaller\begin{tabular}{@{}c@{}}%
\phantom{0}\\\phantom{0}\\3
\end{tabular}\endgroup%
{$\left.\llap{\phantom{%
\begingroup \smaller\smaller\smaller\begin{tabular}{@{}c@{}}%
\phantom{0}\\\phantom{0}\\\phantom{0}
\end{tabular}\endgroup%
}}\!\right]$}%
{$\left[\!\llap{\phantom{%
\begingroup \smaller\smaller\smaller\begin{tabular}{@{}c@{}}%
0\\0\\0
\end{tabular}\endgroup%
}}\right.$}%
\begingroup \smaller\smaller\smaller\begin{tabular}{@{}c@{}}%
4\\-5\\-1
\end{tabular}\endgroup%
\kern3pt%
\begingroup \smaller\smaller\smaller\begin{tabular}{@{}c@{}}%
3\\3\\-1
\end{tabular}\endgroup%
{$\left.\llap{\phantom{%
\begingroup \smaller\smaller\smaller\begin{tabular}{@{}c@{}}%
0\\0\\0
\end{tabular}\endgroup%
}}\!\right]$}%
}%
\ifdim\wd\matricesbox>\halfwidth\myboxwidth=\hsize\else\myboxwidth=\halfwidth\fi
\vbox{%
\ifdim\myboxwidth=\hsize
\setbox\onelinebox=\hbox{%
\vbox{\hbox{%
$\Pi_{4,33}$ spans $L_{7.11}$%
}\hbox{%
$\slashthree2\slashinfty2\rtimes D_{2}$ (shared)%
}%
}%
\hfill\copy\matricesbox
}%
\ifdim\wd\onelinebox>\myboxwidth
\hbox to \myboxwidth{%
$\Pi_{4,33}$ spans $L_{7.11}$%
\hfil
$\slashthree2\slashinfty2\rtimes D_{2}$ (shared)%
}%
\box\matricesbox
\else
\hbox to \myboxwidth{%
\unhbox\onelinebox
}%
\fi
\else
\hbox to \myboxwidth{%
$\Pi_{4,33}$ spans $L_{7.11}$%
\hfil}%
\hbox to \myboxwidth{%
$\slashthree2\slashinfty2\rtimes D_{2}$ (shared)%
\hfil}%
\box\matricesbox
\fi
}%
\hfill\discretionary{}{}{}%
\setbox\matricesbox=\hbox{%
{$\left[\!\llap{\phantom{%
\begingroup \smaller\smaller\smaller\begin{tabular}{@{}c@{}}%
\phantom{0}\\\phantom{0}\\\phantom{0}
\end{tabular}\endgroup%
}}\right.$}%
\begingroup \smaller\smaller\smaller\begin{tabular}{@{}c@{}}%
-1/25\\\phantom{0}\\\phantom{0}
\end{tabular}\endgroup%
\kern3pt%
\begingroup \smaller\smaller\smaller\begin{tabular}{@{}c@{}}%
\phantom{0}\\3/50\\\phantom{0}
\end{tabular}\endgroup%
\kern3pt%
\begingroup \smaller\smaller\smaller\begin{tabular}{@{}c@{}}%
\phantom{0}\\\phantom{0}\\1/2
\end{tabular}\endgroup%
{$\left.\llap{\phantom{%
\begingroup \smaller\smaller\smaller\begin{tabular}{@{}c@{}}%
\phantom{0}\\\phantom{0}\\\phantom{0}
\end{tabular}\endgroup%
}}\!\right]$}%
{$\left[\!\llap{\phantom{%
\begingroup \smaller\smaller\smaller\begin{tabular}{@{}c@{}}%
0\\0\\0
\end{tabular}\endgroup%
}}\right.$}%
\begingroup \smaller\smaller\smaller\begin{tabular}{@{}c@{}}%
6\\7\\3
\end{tabular}\endgroup%
\kern3pt%
\begingroup \smaller\smaller\smaller\begin{tabular}{@{}c@{}}%
1\\-3\\1
\end{tabular}\endgroup%
{$\left.\llap{\phantom{%
\begingroup \smaller\smaller\smaller\begin{tabular}{@{}c@{}}%
0\\0\\0
\end{tabular}\endgroup%
}}\!\right]$}%
}%
\ifdim\wd\matricesbox>\halfwidth\myboxwidth=\hsize\else\myboxwidth=\halfwidth\fi
\vbox{%
\ifdim\myboxwidth=\hsize
\setbox\onelinebox=\hbox{%
\vbox{\hbox{%
$\Pi_{4,34}$ spans $L_{3.2}$%
}\hbox{%
$\slashthree2\slashtwo2\rtimes D_{2}$ (shared)%
}%
}%
\hfill\copy\matricesbox
}%
\ifdim\wd\onelinebox>\myboxwidth
\hbox to \myboxwidth{%
$\Pi_{4,34}$ spans $L_{3.2}$%
\hfil
$\slashthree2\slashtwo2\rtimes D_{2}$ (shared)%
}%
\box\matricesbox
\else
\hbox to \myboxwidth{%
\unhbox\onelinebox
}%
\fi
\else
\hbox to \myboxwidth{%
$\Pi_{4,34}$ spans $L_{3.2}$%
\hfil}%
\hbox to \myboxwidth{%
$\slashthree2\slashtwo2\rtimes D_{2}$ (shared)%
\hfil}%
\box\matricesbox
\fi
}%
\hfill\discretionary{}{}{}%
\setbox\matricesbox=\hbox{%
{$\left[\!\llap{\phantom{%
\begingroup \smaller\smaller\smaller\begin{tabular}{@{}c@{}}%
\phantom{0}\\\phantom{0}\\\phantom{0}
\end{tabular}\endgroup%
}}\right.$}%
\begingroup \smaller\smaller\smaller\begin{tabular}{@{}c@{}}%
-3/26\\\phantom{0}\\\phantom{0}
\end{tabular}\endgroup%
\kern3pt%
\begingroup \smaller\smaller\smaller\begin{tabular}{@{}c@{}}%
\phantom{0}\\3/26\\\phantom{0}
\end{tabular}\endgroup%
\kern3pt%
\begingroup \smaller\smaller\smaller\begin{tabular}{@{}c@{}}%
\phantom{0}\\\phantom{0}\\1/2
\end{tabular}\endgroup%
{$\left.\llap{\phantom{%
\begingroup \smaller\smaller\smaller\begin{tabular}{@{}c@{}}%
\phantom{0}\\\phantom{0}\\\phantom{0}
\end{tabular}\endgroup%
}}\!\right]$}%
{$\left[\!\llap{\phantom{%
\begingroup \smaller\smaller\smaller\begin{tabular}{@{}c@{}}%
0\\0\\0
\end{tabular}\endgroup%
}}\right.$}%
\begingroup \smaller\smaller\smaller\begin{tabular}{@{}c@{}}%
6\\-7\\3
\end{tabular}\endgroup%
\kern3pt%
\begingroup \smaller\smaller\smaller\begin{tabular}{@{}c@{}}%
2\\2\\2
\end{tabular}\endgroup%
{$\left.\llap{\phantom{%
\begingroup \smaller\smaller\smaller\begin{tabular}{@{}c@{}}%
0\\0\\0
\end{tabular}\endgroup%
}}\!\right]$}%
}%
\ifdim\wd\matricesbox>\halfwidth\myboxwidth=\hsize\else\myboxwidth=\halfwidth\fi
\vbox{%
\ifdim\myboxwidth=\hsize
\setbox\onelinebox=\hbox{%
\vbox{\hbox{%
$\Pi_{4,35}$ spans $L_{7.6}$%
}\hbox{%
$\slashthree2\slashinfty2\rtimes D_{2}$ (shared)%
}%
}%
\hfill\copy\matricesbox
}%
\ifdim\wd\onelinebox>\myboxwidth
\hbox to \myboxwidth{%
$\Pi_{4,35}$ spans $L_{7.6}$%
\hfil
$\slashthree2\slashinfty2\rtimes D_{2}$ (shared)%
}%
\box\matricesbox
\else
\hbox to \myboxwidth{%
\unhbox\onelinebox
}%
\fi
\else
\hbox to \myboxwidth{%
$\Pi_{4,35}$ spans $L_{7.6}$%
\hfil}%
\hbox to \myboxwidth{%
$\slashthree2\slashinfty2\rtimes D_{2}$ (shared)%
\hfil}%
\box\matricesbox
\fi
}%
\hfill\discretionary{}{}{}%
\setbox\matricesbox=\hbox{%
{$\left[\!\llap{\phantom{%
\begingroup \smaller\smaller\smaller\begin{tabular}{@{}c@{}}%
\phantom{0}\\\phantom{0}\\\phantom{0}
\end{tabular}\endgroup%
}}\right.$}%
\begingroup \smaller\smaller\smaller\begin{tabular}{@{}c@{}}%
-3/25\\\phantom{0}\\\phantom{0}
\end{tabular}\endgroup%
\kern3pt%
\begingroup \smaller\smaller\smaller\begin{tabular}{@{}c@{}}%
\phantom{0}\\3/25\\\phantom{0}
\end{tabular}\endgroup%
\kern3pt%
\begingroup \smaller\smaller\smaller\begin{tabular}{@{}c@{}}%
\phantom{0}\\\phantom{0}\\1
\end{tabular}\endgroup%
{$\left.\llap{\phantom{%
\begingroup \smaller\smaller\smaller\begin{tabular}{@{}c@{}}%
\phantom{0}\\\phantom{0}\\\phantom{0}
\end{tabular}\endgroup%
}}\!\right]$}%
{$\left[\!\llap{\phantom{%
\begingroup \smaller\smaller\smaller\begin{tabular}{@{}c@{}}%
0\\0\\0
\end{tabular}\endgroup%
}}\right.$}%
\begingroup \smaller\smaller\smaller\begin{tabular}{@{}c@{}}%
12\\13\\-3
\end{tabular}\endgroup%
\kern3pt%
\begingroup \smaller\smaller\smaller\begin{tabular}{@{}c@{}}%
1\\-1\\-1
\end{tabular}\endgroup%
{$\left.\llap{\phantom{%
\begingroup \smaller\smaller\smaller\begin{tabular}{@{}c@{}}%
0\\0\\0
\end{tabular}\endgroup%
}}\!\right]$}%
}%
\ifdim\wd\matricesbox>\halfwidth\myboxwidth=\hsize\else\myboxwidth=\halfwidth\fi
\vbox{%
\ifdim\myboxwidth=\hsize
\setbox\onelinebox=\hbox{%
\vbox{\hbox{%
$\Pi_{4,36}=A_{3,II,\onebar}$ spans $L_{7.7}$%
}\hbox{%
$\slashthree2\slashinfty2\rtimes D_{2}$ (shared)%
}%
}%
\hfill\copy\matricesbox
}%
\ifdim\wd\onelinebox>\myboxwidth
\hbox to \myboxwidth{%
$\Pi_{4,36}=A_{3,II,\onebar}$ spans $L_{7.7}$%
\hfil
$\slashthree2\slashinfty2\rtimes D_{2}$ (shared)%
}%
\box\matricesbox
\else
\hbox to \myboxwidth{%
\unhbox\onelinebox
}%
\fi
\else
\hbox to \myboxwidth{%
$\Pi_{4,36}=A_{3,II,\onebar}$ spans $L_{7.7}$%
\hfil}%
\hbox to \myboxwidth{%
$\slashthree2\slashinfty2\rtimes D_{2}$ (shared)%
\hfil}%
\box\matricesbox
\fi
}%
\hfill\discretionary{}{}{}%
\setbox\matricesbox=\hbox{%
{$\left[\!\llap{\phantom{%
\begingroup \smaller\smaller\smaller\begin{tabular}{@{}c@{}}%
\phantom{0}\\\phantom{0}\\\phantom{0}
\end{tabular}\endgroup%
}}\right.$}%
\begingroup \smaller\smaller\smaller\begin{tabular}{@{}c@{}}%
-4/17\\\phantom{0}\\\phantom{0}
\end{tabular}\endgroup%
\kern3pt%
\begingroup \smaller\smaller\smaller\begin{tabular}{@{}c@{}}%
\phantom{0}\\1/34\\\phantom{0}
\end{tabular}\endgroup%
\kern3pt%
\begingroup \smaller\smaller\smaller\begin{tabular}{@{}c@{}}%
\phantom{0}\\\phantom{0}\\1/2
\end{tabular}\endgroup%
{$\left.\llap{\phantom{%
\begingroup \smaller\smaller\smaller\begin{tabular}{@{}c@{}}%
\phantom{0}\\\phantom{0}\\\phantom{0}
\end{tabular}\endgroup%
}}\!\right]$}%
{$\left[\!\llap{\phantom{%
\begingroup \smaller\smaller\smaller\begin{tabular}{@{}c@{}}%
0\\0\\0
\end{tabular}\endgroup%
}}\right.$}%
\begingroup \smaller\smaller\smaller\begin{tabular}{@{}c@{}}%
4\\14\\-2
\end{tabular}\endgroup%
\kern3pt%
\begingroup \smaller\smaller\smaller\begin{tabular}{@{}c@{}}%
1\\-5\\-1
\end{tabular}\endgroup%
{$\left.\llap{\phantom{%
\begingroup \smaller\smaller\smaller\begin{tabular}{@{}c@{}}%
0\\0\\0
\end{tabular}\endgroup%
}}\!\right]$}%
}%
\ifdim\wd\matricesbox>\halfwidth\myboxwidth=\hsize\else\myboxwidth=\halfwidth\fi
\vbox{%
\ifdim\myboxwidth=\hsize
\setbox\onelinebox=\hbox{%
\vbox{\hbox{%
$\Pi_{4,37}=\hbox{GN}_{20}$ spans $L_{145.1}$%
}\hbox{%
$\infty\slashtwo\infty\slashtwo\rtimes D_{2}$ (shared)%
}%
}%
\hfill\copy\matricesbox
}%
\ifdim\wd\onelinebox>\myboxwidth
\hbox to \myboxwidth{%
$\Pi_{4,37}=\hbox{GN}_{20}$ spans $L_{145.1}$%
\hfil
$\infty\slashtwo\infty\slashtwo\rtimes D_{2}$ (shared)%
}%
\box\matricesbox
\else
\hbox to \myboxwidth{%
\unhbox\onelinebox
}%
\fi
\else
\hbox to \myboxwidth{%
$\Pi_{4,37}=\hbox{GN}_{20}$ spans $L_{145.1}$%
\hfil}%
\hbox to \myboxwidth{%
$\infty\slashtwo\infty\slashtwo\rtimes D_{2}$ (shared)%
\hfil}%
\box\matricesbox
\fi
}%
\hfill\discretionary{}{}{}%
\setbox\matricesbox=\hbox{%
{$\left[\!\llap{\phantom{%
\begingroup \smaller\smaller\smaller\begin{tabular}{@{}c@{}}%
\phantom{0}\\\phantom{0}\\\phantom{0}
\end{tabular}\endgroup%
}}\right.$}%
\begingroup \smaller\smaller\smaller\begin{tabular}{@{}c@{}}%
-1/9\\\phantom{0}\\\phantom{0}
\end{tabular}\endgroup%
\kern3pt%
\begingroup \smaller\smaller\smaller\begin{tabular}{@{}c@{}}%
\phantom{0}\\1/9\\\phantom{0}
\end{tabular}\endgroup%
\kern3pt%
\begingroup \smaller\smaller\smaller\begin{tabular}{@{}c@{}}%
\phantom{0}\\\phantom{0}\\1
\end{tabular}\endgroup%
{$\left.\llap{\phantom{%
\begingroup \smaller\smaller\smaller\begin{tabular}{@{}c@{}}%
\phantom{0}\\\phantom{0}\\\phantom{0}
\end{tabular}\endgroup%
}}\!\right]$}%
{$\left[\!\llap{\phantom{%
\begingroup \smaller\smaller\smaller\begin{tabular}{@{}c@{}}%
0\\0\\0
\end{tabular}\endgroup%
}}\right.$}%
\begingroup \smaller\smaller\smaller\begin{tabular}{@{}c@{}}%
1\\-1\\1
\end{tabular}\endgroup%
\kern3pt%
\begingroup \smaller\smaller\smaller\begin{tabular}{@{}c@{}}%
8\\10\\2
\end{tabular}\endgroup%
{$\left.\llap{\phantom{%
\begingroup \smaller\smaller\smaller\begin{tabular}{@{}c@{}}%
0\\0\\0
\end{tabular}\endgroup%
}}\!\right]$}%
}%
\ifdim\wd\matricesbox>\halfwidth\myboxwidth=\hsize\else\myboxwidth=\halfwidth\fi
\vbox{%
\ifdim\myboxwidth=\hsize
\setbox\onelinebox=\hbox{%
\vbox{\hbox{%
$\Pi_{4,38}=A_{2,II,\onebar}$ spans $L_{1.9}$%
}\hbox{%
$\slashinfty2\slashtwo2\rtimes D_{2}$ (shared)%
}%
}%
\hfill\copy\matricesbox
}%
\ifdim\wd\onelinebox>\myboxwidth
\hbox to \myboxwidth{%
$\Pi_{4,38}=A_{2,II,\onebar}$ spans $L_{1.9}$%
\hfil
$\slashinfty2\slashtwo2\rtimes D_{2}$ (shared)%
}%
\box\matricesbox
\else
\hbox to \myboxwidth{%
\unhbox\onelinebox
}%
\fi
\else
\hbox to \myboxwidth{%
$\Pi_{4,38}=A_{2,II,\onebar}$ spans $L_{1.9}$%
\hfil}%
\hbox to \myboxwidth{%
$\slashinfty2\slashtwo2\rtimes D_{2}$ (shared)%
\hfil}%
\box\matricesbox
\fi
}%
\hfill\discretionary{}{}{}%
\setbox\matricesbox=\hbox{%
{$\left[\!\llap{\phantom{%
\begingroup \smaller\smaller\smaller\begin{tabular}{@{}c@{}}%
\phantom{0}\\\phantom{0}\\\phantom{0}
\end{tabular}\endgroup%
}}\right.$}%
\begingroup \smaller\smaller\smaller\begin{tabular}{@{}c@{}}%
-1/8\\\phantom{0}\\\phantom{0}
\end{tabular}\endgroup%
\kern3pt%
\begingroup \smaller\smaller\smaller\begin{tabular}{@{}c@{}}%
\phantom{0}\\1/8\\\phantom{0}
\end{tabular}\endgroup%
\kern3pt%
\begingroup \smaller\smaller\smaller\begin{tabular}{@{}c@{}}%
\phantom{0}\\\phantom{0}\\1
\end{tabular}\endgroup%
{$\left.\llap{\phantom{%
\begingroup \smaller\smaller\smaller\begin{tabular}{@{}c@{}}%
\phantom{0}\\\phantom{0}\\\phantom{0}
\end{tabular}\endgroup%
}}\!\right]$}%
{$\left[\!\llap{\phantom{%
\begingroup \smaller\smaller\smaller\begin{tabular}{@{}c@{}}%
0\\0\\0
\end{tabular}\endgroup%
}}\right.$}%
\begingroup \smaller\smaller\smaller\begin{tabular}{@{}c@{}}%
1\\-1\\1
\end{tabular}\endgroup%
\kern3pt%
\begingroup \smaller\smaller\smaller\begin{tabular}{@{}c@{}}%
16\\16\\4
\end{tabular}\endgroup%
{$\left.\llap{\phantom{%
\begingroup \smaller\smaller\smaller\begin{tabular}{@{}c@{}}%
0\\0\\0
\end{tabular}\endgroup%
}}\!\right]$}%
}%
\ifdim\wd\matricesbox>\halfwidth\myboxwidth=\hsize\else\myboxwidth=\halfwidth\fi
\vbox{%
\ifdim\myboxwidth=\hsize
\setbox\onelinebox=\hbox{%
\vbox{\hbox{%
$\Pi_{4,39}=A_{4,II,\onebar}=\hbox{GN}_{15}$ spans $L_{140.4}$%
}\hbox{%
$\slashinfty2\slashinfty2\rtimes D_{2}$ (shared)%
}%
}%
\hfill\copy\matricesbox
}%
\ifdim\wd\onelinebox>\myboxwidth
\hbox to \myboxwidth{%
$\Pi_{4,39}=A_{4,II,\onebar}=\hbox{GN}_{15}$ spans $L_{140.4}$%
\hfil
$\slashinfty2\slashinfty2\rtimes D_{2}$ (shared)%
}%
\box\matricesbox
\else
\hbox to \myboxwidth{%
\unhbox\onelinebox
}%
\fi
\else
\hbox to \myboxwidth{%
$\Pi_{4,39}=A_{4,II,\onebar}=\hbox{GN}_{15}$ spans $L_{140.4}$%
\hfil}%
\hbox to \myboxwidth{%
$\slashinfty2\slashinfty2\rtimes D_{2}$ (shared)%
\hfil}%
\box\matricesbox
\fi
}%
\hfill\discretionary{}{}{}%
\setbox\matricesbox=\hbox{%
{$\left[\!\llap{\phantom{%
\begingroup \smaller\smaller\smaller\begin{tabular}{@{}c@{}}%
\phantom{0}\\\phantom{0}\\\phantom{0}
\end{tabular}\endgroup%
}}\right.$}%
\begingroup \smaller\smaller\smaller\begin{tabular}{@{}c@{}}%
-1/5\\\phantom{0}\\\phantom{0}
\end{tabular}\endgroup%
\kern3pt%
\begingroup \smaller\smaller\smaller\begin{tabular}{@{}c@{}}%
\phantom{0}\\1/5\\\phantom{0}
\end{tabular}\endgroup%
\kern3pt%
\begingroup \smaller\smaller\smaller\begin{tabular}{@{}c@{}}%
\phantom{0}\\\phantom{0}\\1
\end{tabular}\endgroup%
{$\left.\llap{\phantom{%
\begingroup \smaller\smaller\smaller\begin{tabular}{@{}c@{}}%
\phantom{0}\\\phantom{0}\\\phantom{0}
\end{tabular}\endgroup%
}}\!\right]$}%
{$\left[\!\llap{\phantom{%
\begingroup \smaller\smaller\smaller\begin{tabular}{@{}c@{}}%
0\\0\\0
\end{tabular}\endgroup%
}}\right.$}%
\begingroup \smaller\smaller\smaller\begin{tabular}{@{}c@{}}%
1\\1\\-1
\end{tabular}\endgroup%
\kern3pt%
\begingroup \smaller\smaller\smaller\begin{tabular}{@{}c@{}}%
2\\-3\\-1
\end{tabular}\endgroup%
{$\left.\llap{\phantom{%
\begingroup \smaller\smaller\smaller\begin{tabular}{@{}c@{}}%
0\\0\\0
\end{tabular}\endgroup%
}}\!\right]$}%
}%
\ifdim\wd\matricesbox>\halfwidth\myboxwidth=\hsize\else\myboxwidth=\halfwidth\fi
\vbox{%
\ifdim\myboxwidth=\hsize
\setbox\onelinebox=\hbox{%
\vbox{\hbox{%
$\Pi_{4,40}$ spans $L_{1.7}$%
}\hbox{%
$\slashinfty2\slashtwo2\rtimes D_{2}$ (shared)%
}%
}%
\hfill\copy\matricesbox
}%
\ifdim\wd\onelinebox>\myboxwidth
\hbox to \myboxwidth{%
$\Pi_{4,40}$ spans $L_{1.7}$%
\hfil
$\slashinfty2\slashtwo2\rtimes D_{2}$ (shared)%
}%
\box\matricesbox
\else
\hbox to \myboxwidth{%
\unhbox\onelinebox
}%
\fi
\else
\hbox to \myboxwidth{%
$\Pi_{4,40}$ spans $L_{1.7}$%
\hfil}%
\hbox to \myboxwidth{%
$\slashinfty2\slashtwo2\rtimes D_{2}$ (shared)%
\hfil}%
\box\matricesbox
\fi
}%
\hfill\discretionary{}{}{}%
\setbox\matricesbox=\hbox{%
{$\left[\!\llap{\phantom{%
\begingroup \smaller\smaller\smaller\begin{tabular}{@{}c@{}}%
\phantom{0}\\\phantom{0}\\\phantom{0}
\end{tabular}\endgroup%
}}\right.$}%
\begingroup \smaller\smaller\smaller\begin{tabular}{@{}c@{}}%
-1/4\\\phantom{0}\\\phantom{0}
\end{tabular}\endgroup%
\kern3pt%
\begingroup \smaller\smaller\smaller\begin{tabular}{@{}c@{}}%
\phantom{0}\\1/4\\\phantom{0}
\end{tabular}\endgroup%
\kern3pt%
\begingroup \smaller\smaller\smaller\begin{tabular}{@{}c@{}}%
\phantom{0}\\\phantom{0}\\1
\end{tabular}\endgroup%
{$\left.\llap{\phantom{%
\begingroup \smaller\smaller\smaller\begin{tabular}{@{}c@{}}%
\phantom{0}\\\phantom{0}\\\phantom{0}
\end{tabular}\endgroup%
}}\!\right]$}%
{$\left[\!\llap{\phantom{%
\begingroup \smaller\smaller\smaller\begin{tabular}{@{}c@{}}%
0\\0\\0
\end{tabular}\endgroup%
}}\right.$}%
\begingroup \smaller\smaller\smaller\begin{tabular}{@{}c@{}}%
1\\-1\\1
\end{tabular}\endgroup%
\kern3pt%
\begingroup \smaller\smaller\smaller\begin{tabular}{@{}c@{}}%
4\\4\\2
\end{tabular}\endgroup%
{$\left.\llap{\phantom{%
\begingroup \smaller\smaller\smaller\begin{tabular}{@{}c@{}}%
0\\0\\0
\end{tabular}\endgroup%
}}\!\right]$}%
}%
\ifdim\wd\matricesbox>\halfwidth\myboxwidth=\hsize\else\myboxwidth=\halfwidth\fi
\vbox{%
\ifdim\myboxwidth=\hsize
\setbox\onelinebox=\hbox{%
\vbox{\hbox{%
$\Pi_{4,41}=\hbox{GN}_{23}$ spans $L_{140.3}$%
}\hbox{%
$\slashinfty2\slashinfty2\rtimes D_{2}$ (shared)%
}%
}%
\hfill\copy\matricesbox
}%
\ifdim\wd\onelinebox>\myboxwidth
\hbox to \myboxwidth{%
$\Pi_{4,41}=\hbox{GN}_{23}$ spans $L_{140.3}$%
\hfil
$\slashinfty2\slashinfty2\rtimes D_{2}$ (shared)%
}%
\box\matricesbox
\else
\hbox to \myboxwidth{%
\unhbox\onelinebox
}%
\fi
\else
\hbox to \myboxwidth{%
$\Pi_{4,41}=\hbox{GN}_{23}$ spans $L_{140.3}$%
\hfil}%
\hbox to \myboxwidth{%
$\slashinfty2\slashinfty2\rtimes D_{2}$ (shared)%
\hfil}%
\box\matricesbox
\fi
}%
\hfill\discretionary{}{}{}%
\setbox\matricesbox=\hbox{%
{$\left[\!\llap{\phantom{%
\begingroup \smaller\smaller\smaller\begin{tabular}{@{}c@{}}%
\phantom{0}\\\phantom{0}\\\phantom{0}
\end{tabular}\endgroup%
}}\right.$}%
\begingroup \smaller\smaller\smaller\begin{tabular}{@{}c@{}}%
-1/6\\\phantom{0}\\\phantom{0}
\end{tabular}\endgroup%
\kern3pt%
\begingroup \smaller\smaller\smaller\begin{tabular}{@{}c@{}}%
\phantom{0}\\2/3\\\phantom{0}
\end{tabular}\endgroup%
\kern3pt%
\begingroup \smaller\smaller\smaller\begin{tabular}{@{}c@{}}%
\phantom{0}\\\phantom{0}\\1/2
\end{tabular}\endgroup%
{$\left.\llap{\phantom{%
\begingroup \smaller\smaller\smaller\begin{tabular}{@{}c@{}}%
\phantom{0}\\\phantom{0}\\\phantom{0}
\end{tabular}\endgroup%
}}\!\right]$}%
{$\left[\!\llap{\phantom{%
\begingroup \smaller\smaller\smaller\begin{tabular}{@{}c@{}}%
0\\0\\0
\end{tabular}\endgroup%
}}\right.$}%
\begingroup \smaller\smaller\smaller\begin{tabular}{@{}c@{}}%
2\\1\\2
\end{tabular}\endgroup%
\kern3pt%
\begingroup \smaller\smaller\smaller\begin{tabular}{@{}c@{}}%
1\\-1\\1
\end{tabular}\endgroup%
{$\left.\llap{\phantom{%
\begingroup \smaller\smaller\smaller\begin{tabular}{@{}c@{}}%
0\\0\\0
\end{tabular}\endgroup%
}}\!\right]$}%
}%
\ifdim\wd\matricesbox>\halfwidth\myboxwidth=\hsize\else\myboxwidth=\halfwidth\fi
\vbox{%
\ifdim\myboxwidth=\hsize
\setbox\onelinebox=\hbox{%
\vbox{\hbox{%
$\Pi_{4,42}$ spans $L_{1.5}$%
}\hbox{%
$\slashinfty2\slashtwo2\rtimes D_{2}$ (shared)%
}%
}%
\hfill\copy\matricesbox
}%
\ifdim\wd\onelinebox>\myboxwidth
\hbox to \myboxwidth{%
$\Pi_{4,42}$ spans $L_{1.5}$%
\hfil
$\slashinfty2\slashtwo2\rtimes D_{2}$ (shared)%
}%
\box\matricesbox
\else
\hbox to \myboxwidth{%
\unhbox\onelinebox
}%
\fi
\else
\hbox to \myboxwidth{%
$\Pi_{4,42}$ spans $L_{1.5}$%
\hfil}%
\hbox to \myboxwidth{%
$\slashinfty2\slashtwo2\rtimes D_{2}$ (shared)%
\hfil}%
\box\matricesbox
\fi
}%
\hfill\discretionary{}{}{}%
\setbox\matricesbox=\hbox{%
{$\left[\!\llap{\phantom{%
\begingroup \smaller\smaller\smaller\begin{tabular}{@{}c@{}}%
\phantom{0}\\\phantom{0}\\\phantom{0}
\end{tabular}\endgroup%
}}\right.$}%
\begingroup \smaller\smaller\smaller\begin{tabular}{@{}c@{}}%
-1/16\\\phantom{0}\\\phantom{0}
\end{tabular}\endgroup%
\kern3pt%
\begingroup \smaller\smaller\smaller\begin{tabular}{@{}c@{}}%
\phantom{0}\\1/16\\\phantom{0}
\end{tabular}\endgroup%
\kern3pt%
\begingroup \smaller\smaller\smaller\begin{tabular}{@{}c@{}}%
\phantom{0}\\\phantom{0}\\1/2
\end{tabular}\endgroup%
{$\left.\llap{\phantom{%
\begingroup \smaller\smaller\smaller\begin{tabular}{@{}c@{}}%
\phantom{0}\\\phantom{0}\\\phantom{0}
\end{tabular}\endgroup%
}}\!\right]$}%
{$\left[\!\llap{\phantom{%
\begingroup \smaller\smaller\smaller\begin{tabular}{@{}c@{}}%
0\\0\\0
\end{tabular}\endgroup%
}}\right.$}%
\begingroup \smaller\smaller\smaller\begin{tabular}{@{}c@{}}%
8\\-8\\4
\end{tabular}\endgroup%
\kern3pt%
\begingroup \smaller\smaller\smaller\begin{tabular}{@{}c@{}}%
1\\3\\1
\end{tabular}\endgroup%
{$\left.\llap{\phantom{%
\begingroup \smaller\smaller\smaller\begin{tabular}{@{}c@{}}%
0\\0\\0
\end{tabular}\endgroup%
}}\!\right]$}%
}%
\ifdim\wd\matricesbox>\halfwidth\myboxwidth=\hsize\else\myboxwidth=\halfwidth\fi
\vbox{%
\ifdim\myboxwidth=\hsize
\setbox\onelinebox=\hbox{%
\vbox{\hbox{%
$\Pi_{4,43}$ spans $L_{1.3}$%
}\hbox{%
$\slashinfty2\slashtwo2\rtimes D_{2}$ (shared)%
}%
}%
\hfill\copy\matricesbox
}%
\ifdim\wd\onelinebox>\myboxwidth
\hbox to \myboxwidth{%
$\Pi_{4,43}$ spans $L_{1.3}$%
\hfil
$\slashinfty2\slashtwo2\rtimes D_{2}$ (shared)%
}%
\box\matricesbox
\else
\hbox to \myboxwidth{%
\unhbox\onelinebox
}%
\fi
\else
\hbox to \myboxwidth{%
$\Pi_{4,43}$ spans $L_{1.3}$%
\hfil}%
\hbox to \myboxwidth{%
$\slashinfty2\slashtwo2\rtimes D_{2}$ (shared)%
\hfil}%
\box\matricesbox
\fi
}%
\hfill\discretionary{}{}{}%
\setbox\matricesbox=\hbox{%
{$\left[\!\llap{\phantom{%
\begingroup \smaller\smaller\smaller\begin{tabular}{@{}c@{}}%
\phantom{0}\\\phantom{0}\\\phantom{0}
\end{tabular}\endgroup%
}}\right.$}%
\begingroup \smaller\smaller\smaller\begin{tabular}{@{}c@{}}%
-1/8\\\phantom{0}\\\phantom{0}
\end{tabular}\endgroup%
\kern3pt%
\begingroup \smaller\smaller\smaller\begin{tabular}{@{}c@{}}%
\phantom{0}\\5/2\\-1
\end{tabular}\endgroup%
\kern3pt%
\begingroup \smaller\smaller\smaller\begin{tabular}{@{}c@{}}%
\phantom{0}\\-1\\6
\end{tabular}\endgroup%
{$\left.\llap{\phantom{%
\begingroup \smaller\smaller\smaller\begin{tabular}{@{}c@{}}%
\phantom{0}\\\phantom{0}\\\phantom{0}
\end{tabular}\endgroup%
}}\!\right]$}%
{$\left[\!\llap{\phantom{%
\begingroup \smaller\smaller\smaller\begin{tabular}{@{}c@{}}%
0\\0\\0
\end{tabular}\endgroup%
}}\right.$}%
\begingroup \smaller\smaller\smaller\begin{tabular}{@{}c@{}}%
2\\1\\0
\end{tabular}\endgroup%
\kern3pt%
\begingroup \smaller\smaller\smaller\begin{tabular}{@{}c@{}}%
4\\0\\1
\end{tabular}\endgroup%
{$\left.\llap{\phantom{%
\begingroup \smaller\smaller\smaller\begin{tabular}{@{}c@{}}%
0\\0\\0
\end{tabular}\endgroup%
}}\!\right]$}%
}%
\ifdim\wd\matricesbox>\halfwidth\myboxwidth=\hsize\else\myboxwidth=\halfwidth\fi
\vbox{%
\ifdim\myboxwidth=\hsize
\setbox\onelinebox=\hbox{%
\vbox{\hbox{%
$\Pi_{4,44}$ spans $L_{8.1}$%
}\hbox{%
$4242\rtimes C_{2}$%
}%
}%
\hfill\copy\matricesbox
}%
\ifdim\wd\onelinebox>\myboxwidth
\hbox to \myboxwidth{%
$\Pi_{4,44}$ spans $L_{8.1}$%
\hfil
$4242\rtimes C_{2}$%
}%
\box\matricesbox
\else
\hbox to \myboxwidth{%
\unhbox\onelinebox
}%
\fi
\else
\hbox to \myboxwidth{%
$\Pi_{4,44}$ spans $L_{8.1}$%
\hfil}%
\hbox to \myboxwidth{%
$4242\rtimes C_{2}$%
\hfil}%
\box\matricesbox
\fi
}%
\hfill\discretionary{}{}{}%
\setbox\matricesbox=\hbox{%
{$\left[\!\llap{\phantom{%
\begingroup \smaller\smaller\smaller\begin{tabular}{@{}c@{}}%
\phantom{0}\\\phantom{0}\\\phantom{0}
\end{tabular}\endgroup%
}}\right.$}%
\begingroup \smaller\smaller\smaller\begin{tabular}{@{}c@{}}%
-1/8\\\phantom{0}\\\phantom{0}
\end{tabular}\endgroup%
\kern3pt%
\begingroup \smaller\smaller\smaller\begin{tabular}{@{}c@{}}%
\phantom{0}\\5/2\\-1
\end{tabular}\endgroup%
\kern3pt%
\begingroup \smaller\smaller\smaller\begin{tabular}{@{}c@{}}%
\phantom{0}\\-1\\10
\end{tabular}\endgroup%
{$\left.\llap{\phantom{%
\begingroup \smaller\smaller\smaller\begin{tabular}{@{}c@{}}%
\phantom{0}\\\phantom{0}\\\phantom{0}
\end{tabular}\endgroup%
}}\!\right]$}%
{$\left[\!\llap{\phantom{%
\begingroup \smaller\smaller\smaller\begin{tabular}{@{}c@{}}%
0\\0\\0
\end{tabular}\endgroup%
}}\right.$}%
\begingroup \smaller\smaller\smaller\begin{tabular}{@{}c@{}}%
2\\1\\0
\end{tabular}\endgroup%
\kern3pt%
\begingroup \smaller\smaller\smaller\begin{tabular}{@{}c@{}}%
6\\-1\\-1
\end{tabular}\endgroup%
{$\left.\llap{\phantom{%
\begingroup \smaller\smaller\smaller\begin{tabular}{@{}c@{}}%
0\\0\\0
\end{tabular}\endgroup%
}}\!\right]$}%
}%
\ifdim\wd\matricesbox>\halfwidth\myboxwidth=\hsize\else\myboxwidth=\halfwidth\fi
\vbox{%
\ifdim\myboxwidth=\hsize
\setbox\onelinebox=\hbox{%
\vbox{\hbox{%
$\Pi_{4,45}$ spans $L_{168.2}$%
}\hbox{%
$6262\rtimes C_{2}$%
}%
}%
\hfill\copy\matricesbox
}%
\ifdim\wd\onelinebox>\myboxwidth
\hbox to \myboxwidth{%
$\Pi_{4,45}$ spans $L_{168.2}$%
\hfil
$6262\rtimes C_{2}$%
}%
\box\matricesbox
\else
\hbox to \myboxwidth{%
\unhbox\onelinebox
}%
\fi
\else
\hbox to \myboxwidth{%
$\Pi_{4,45}$ spans $L_{168.2}$%
\hfil}%
\hbox to \myboxwidth{%
$6262\rtimes C_{2}$%
\hfil}%
\box\matricesbox
\fi
}%
\hfill\discretionary{}{}{}%
\setbox\matricesbox=\hbox{%
{$\left[\!\llap{\phantom{%
\begingroup \smaller\smaller\smaller\begin{tabular}{@{}c@{}}%
\phantom{0}\\\phantom{0}\\\phantom{0}
\end{tabular}\endgroup%
}}\right.$}%
\begingroup \smaller\smaller\smaller\begin{tabular}{@{}c@{}}%
-1/8\\\phantom{0}\\\phantom{0}
\end{tabular}\endgroup%
\kern3pt%
\begingroup \smaller\smaller\smaller\begin{tabular}{@{}c@{}}%
\phantom{0}\\5/2\\-1/2
\end{tabular}\endgroup%
\kern3pt%
\begingroup \smaller\smaller\smaller\begin{tabular}{@{}c@{}}%
\phantom{0}\\-1/2\\29/2
\end{tabular}\endgroup%
{$\left.\llap{\phantom{%
\begingroup \smaller\smaller\smaller\begin{tabular}{@{}c@{}}%
\phantom{0}\\\phantom{0}\\\phantom{0}
\end{tabular}\endgroup%
}}\!\right]$}%
{$\left[\!\llap{\phantom{%
\begingroup \smaller\smaller\smaller\begin{tabular}{@{}c@{}}%
0\\0\\0
\end{tabular}\endgroup%
}}\right.$}%
\begingroup \smaller\smaller\smaller\begin{tabular}{@{}c@{}}%
2\\1\\0
\end{tabular}\endgroup%
\kern3pt%
\begingroup \smaller\smaller\smaller\begin{tabular}{@{}c@{}}%
8\\-1\\-1
\end{tabular}\endgroup%
{$\left.\llap{\phantom{%
\begingroup \smaller\smaller\smaller\begin{tabular}{@{}c@{}}%
0\\0\\0
\end{tabular}\endgroup%
}}\!\right]$}%
}%
\ifdim\wd\matricesbox>\halfwidth\myboxwidth=\hsize\else\myboxwidth=\halfwidth\fi
\vbox{%
\ifdim\myboxwidth=\hsize
\setbox\onelinebox=\hbox{%
\vbox{\hbox{%
$\Pi_{4,46}=\hbox{GN}_{22}$ spans $L_{148.2}$%
}\hbox{%
$\infty2\infty2\rtimes C_{2}$ (shared)%
}%
}%
\hfill\copy\matricesbox
}%
\ifdim\wd\onelinebox>\myboxwidth
\hbox to \myboxwidth{%
$\Pi_{4,46}=\hbox{GN}_{22}$ spans $L_{148.2}$%
\hfil
$\infty2\infty2\rtimes C_{2}$ (shared)%
}%
\box\matricesbox
\else
\hbox to \myboxwidth{%
\unhbox\onelinebox
}%
\fi
\else
\hbox to \myboxwidth{%
$\Pi_{4,46}=\hbox{GN}_{22}$ spans $L_{148.2}$%
\hfil}%
\hbox to \myboxwidth{%
$\infty2\infty2\rtimes C_{2}$ (shared)%
\hfil}%
\box\matricesbox
\fi
}%
\hfill\discretionary{}{}{}%
\setbox\matricesbox=\hbox{%
{$\left[\!\llap{\phantom{%
\begingroup \smaller\smaller\smaller\begin{tabular}{@{}c@{}}%
\phantom{0}\\\phantom{0}\\\phantom{0}
\end{tabular}\endgroup%
}}\right.$}%
\begingroup \smaller\smaller\smaller\begin{tabular}{@{}c@{}}%
-1/28\\\phantom{0}\\\phantom{0}
\end{tabular}\endgroup%
\kern3pt%
\begingroup \smaller\smaller\smaller\begin{tabular}{@{}c@{}}%
\phantom{0}\\15/7\\-6/7
\end{tabular}\endgroup%
\kern3pt%
\begingroup \smaller\smaller\smaller\begin{tabular}{@{}c@{}}%
\phantom{0}\\-6/7\\15/7
\end{tabular}\endgroup%
{$\left.\llap{\phantom{%
\begingroup \smaller\smaller\smaller\begin{tabular}{@{}c@{}}%
\phantom{0}\\\phantom{0}\\\phantom{0}
\end{tabular}\endgroup%
}}\!\right]$}%
{$\left[\!\llap{\phantom{%
\begingroup \smaller\smaller\smaller\begin{tabular}{@{}c@{}}%
0\\0\\0
\end{tabular}\endgroup%
}}\right.$}%
\begingroup \smaller\smaller\smaller\begin{tabular}{@{}c@{}}%
18\\-1\\-4
\end{tabular}\endgroup%
\kern3pt%
\begingroup \smaller\smaller\smaller\begin{tabular}{@{}c@{}}%
18\\-4\\-1
\end{tabular}\endgroup%
\kern3pt%
\begingroup \smaller\smaller\smaller\begin{tabular}{@{}c@{}}%
6\\1\\2
\end{tabular}\endgroup%
\kern3pt%
\begingroup \smaller\smaller\smaller\begin{tabular}{@{}c@{}}%
2\\1\\0
\end{tabular}\endgroup%
{$\left.\llap{\phantom{%
\begingroup \smaller\smaller\smaller\begin{tabular}{@{}c@{}}%
0\\0\\0
\end{tabular}\endgroup%
}}\!\right]$}%
}%
\ifdim\wd\matricesbox>\halfwidth\myboxwidth=\hsize\else\myboxwidth=\halfwidth\fi
\vbox{%
\ifdim\myboxwidth=\hsize
\setbox\onelinebox=\hbox{%
\vbox{\hbox{%
$\Pi_{4,47}$ spans $L_{155.1}$%
}\hbox{%
$3622$ (shared)%
}%
}%
\hfill\copy\matricesbox
}%
\ifdim\wd\onelinebox>\myboxwidth
\hbox to \myboxwidth{%
$\Pi_{4,47}$ spans $L_{155.1}$%
\hfil
$3622$ (shared)%
}%
\box\matricesbox
\else
\hbox to \myboxwidth{%
\unhbox\onelinebox
}%
\fi
\else
\hbox to \myboxwidth{%
$\Pi_{4,47}$ spans $L_{155.1}$%
\hfil}%
\hbox to \myboxwidth{%
$3622$ (shared)%
\hfil}%
\box\matricesbox
\fi
}%
\hfill\discretionary{}{}{}%
\setbox\matricesbox=\hbox{%
{$\left[\!\llap{\phantom{%
\begingroup \smaller\smaller\smaller\begin{tabular}{@{}c@{}}%
\phantom{0}\\\phantom{0}\\\phantom{0}
\end{tabular}\endgroup%
}}\right.$}%
\begingroup \smaller\smaller\smaller\begin{tabular}{@{}c@{}}%
-3/25\\\phantom{0}\\\phantom{0}
\end{tabular}\endgroup%
\kern3pt%
\begingroup \smaller\smaller\smaller\begin{tabular}{@{}c@{}}%
\phantom{0}\\12/25\\-6/25
\end{tabular}\endgroup%
\kern3pt%
\begingroup \smaller\smaller\smaller\begin{tabular}{@{}c@{}}%
\phantom{0}\\-6/25\\28/25
\end{tabular}\endgroup%
{$\left.\llap{\phantom{%
\begingroup \smaller\smaller\smaller\begin{tabular}{@{}c@{}}%
\phantom{0}\\\phantom{0}\\\phantom{0}
\end{tabular}\endgroup%
}}\!\right]$}%
{$\left[\!\llap{\phantom{%
\begingroup \smaller\smaller\smaller\begin{tabular}{@{}c@{}}%
0\\0\\0
\end{tabular}\endgroup%
}}\right.$}%
\begingroup \smaller\smaller\smaller\begin{tabular}{@{}c@{}}%
12\\5\\-3
\end{tabular}\endgroup%
\kern3pt%
\begingroup \smaller\smaller\smaller\begin{tabular}{@{}c@{}}%
12\\8\\3
\end{tabular}\endgroup%
\kern3pt%
\begingroup \smaller\smaller\smaller\begin{tabular}{@{}c@{}}%
4\\-1\\2
\end{tabular}\endgroup%
\kern3pt%
\begingroup \smaller\smaller\smaller\begin{tabular}{@{}c@{}}%
1\\-1\\-1
\end{tabular}\endgroup%
{$\left.\llap{\phantom{%
\begingroup \smaller\smaller\smaller\begin{tabular}{@{}c@{}}%
0\\0\\0
\end{tabular}\endgroup%
}}\!\right]$}%
}%
\ifdim\wd\matricesbox>\halfwidth\myboxwidth=\hsize\else\myboxwidth=\halfwidth\fi
\vbox{%
\ifdim\myboxwidth=\hsize
\setbox\onelinebox=\hbox{%
\vbox{\hbox{%
$\Pi_{4,48}$ spans $L_{221.3}$%
}\hbox{%
$36\infty2$ (shared)%
}%
}%
\hfill\copy\matricesbox
}%
\ifdim\wd\onelinebox>\myboxwidth
\hbox to \myboxwidth{%
$\Pi_{4,48}$ spans $L_{221.3}$%
\hfil
$36\infty2$ (shared)%
}%
\box\matricesbox
\else
\hbox to \myboxwidth{%
\unhbox\onelinebox
}%
\fi
\else
\hbox to \myboxwidth{%
$\Pi_{4,48}$ spans $L_{221.3}$%
\hfil}%
\hbox to \myboxwidth{%
$36\infty2$ (shared)%
\hfil}%
\box\matricesbox
\fi
}%
\hfill\discretionary{}{}{}%
\setbox\matricesbox=\hbox{%
{$\left[\!\llap{\phantom{%
\begingroup \smaller\smaller\smaller\begin{tabular}{@{}c@{}}%
\phantom{0}\\\phantom{0}\\\phantom{0}
\end{tabular}\endgroup%
}}\right.$}%
\begingroup \smaller\smaller\smaller\begin{tabular}{@{}c@{}}%
-1/8\\\phantom{0}\\\phantom{0}
\end{tabular}\endgroup%
\kern3pt%
\begingroup \smaller\smaller\smaller\begin{tabular}{@{}c@{}}%
\phantom{0}\\5/2\\-1/2
\end{tabular}\endgroup%
\kern3pt%
\begingroup \smaller\smaller\smaller\begin{tabular}{@{}c@{}}%
\phantom{0}\\-1/2\\5/2
\end{tabular}\endgroup%
{$\left.\llap{\phantom{%
\begingroup \smaller\smaller\smaller\begin{tabular}{@{}c@{}}%
\phantom{0}\\\phantom{0}\\\phantom{0}
\end{tabular}\endgroup%
}}\!\right]$}%
{$\left[\!\llap{\phantom{%
\begingroup \smaller\smaller\smaller\begin{tabular}{@{}c@{}}%
0\\0\\0
\end{tabular}\endgroup%
}}\right.$}%
\begingroup \smaller\smaller\smaller\begin{tabular}{@{}c@{}}%
2\\0\\1
\end{tabular}\endgroup%
\kern3pt%
\begingroup \smaller\smaller\smaller\begin{tabular}{@{}c@{}}%
2\\1\\0
\end{tabular}\endgroup%
\kern3pt%
\begingroup \smaller\smaller\smaller\begin{tabular}{@{}c@{}}%
6\\-1\\-2
\end{tabular}\endgroup%
\kern3pt%
\begingroup \smaller\smaller\smaller\begin{tabular}{@{}c@{}}%
2\\-1\\0
\end{tabular}\endgroup%
{$\left.\llap{\phantom{%
\begingroup \smaller\smaller\smaller\begin{tabular}{@{}c@{}}%
0\\0\\0
\end{tabular}\endgroup%
}}\!\right]$}%
}%
\ifdim\wd\matricesbox>\halfwidth\myboxwidth=\hsize\else\myboxwidth=\halfwidth\fi
\vbox{%
\ifdim\myboxwidth=\hsize
\setbox\onelinebox=\hbox{%
\vbox{\hbox{%
$\Pi_{4,49}$ spans $L_{3.1}$%
}\hbox{%
$3622$ (shared)%
}%
}%
\hfill\copy\matricesbox
}%
\ifdim\wd\onelinebox>\myboxwidth
\hbox to \myboxwidth{%
$\Pi_{4,49}$ spans $L_{3.1}$%
\hfil
$3622$ (shared)%
}%
\box\matricesbox
\else
\hbox to \myboxwidth{%
\unhbox\onelinebox
}%
\fi
\else
\hbox to \myboxwidth{%
$\Pi_{4,49}$ spans $L_{3.1}$%
\hfil}%
\hbox to \myboxwidth{%
$3622$ (shared)%
\hfil}%
\box\matricesbox
\fi
}%
\hfill\discretionary{}{}{}%
\setbox\matricesbox=\hbox{%
{$\left[\!\llap{\phantom{%
\begingroup \smaller\smaller\smaller\begin{tabular}{@{}c@{}}%
\phantom{0}\\\phantom{0}\\\phantom{0}
\end{tabular}\endgroup%
}}\right.$}%
\begingroup \smaller\smaller\smaller\begin{tabular}{@{}c@{}}%
-1/9\\\phantom{0}\\\phantom{0}
\end{tabular}\endgroup%
\kern3pt%
\begingroup \smaller\smaller\smaller\begin{tabular}{@{}c@{}}%
\phantom{0}\\4/9\\-2/9
\end{tabular}\endgroup%
\kern3pt%
\begingroup \smaller\smaller\smaller\begin{tabular}{@{}c@{}}%
\phantom{0}\\-2/9\\28/9
\end{tabular}\endgroup%
{$\left.\llap{\phantom{%
\begingroup \smaller\smaller\smaller\begin{tabular}{@{}c@{}}%
\phantom{0}\\\phantom{0}\\\phantom{0}
\end{tabular}\endgroup%
}}\!\right]$}%
{$\left[\!\llap{\phantom{%
\begingroup \smaller\smaller\smaller\begin{tabular}{@{}c@{}}%
0\\0\\0
\end{tabular}\endgroup%
}}\right.$}%
\begingroup \smaller\smaller\smaller\begin{tabular}{@{}c@{}}%
4\\-3\\-1
\end{tabular}\endgroup%
\kern3pt%
\begingroup \smaller\smaller\smaller\begin{tabular}{@{}c@{}}%
4\\-2\\1
\end{tabular}\endgroup%
\kern3pt%
\begingroup \smaller\smaller\smaller\begin{tabular}{@{}c@{}}%
12\\7\\2
\end{tabular}\endgroup%
\kern3pt%
\begingroup \smaller\smaller\smaller\begin{tabular}{@{}c@{}}%
3\\1\\-1
\end{tabular}\endgroup%
{$\left.\llap{\phantom{%
\begingroup \smaller\smaller\smaller\begin{tabular}{@{}c@{}}%
0\\0\\0
\end{tabular}\endgroup%
}}\!\right]$}%
}%
\ifdim\wd\matricesbox>\halfwidth\myboxwidth=\hsize\else\myboxwidth=\halfwidth\fi
\vbox{%
\ifdim\myboxwidth=\hsize
\setbox\onelinebox=\hbox{%
\vbox{\hbox{%
$\Pi_{4,50}$ spans $L_{7.11}$%
}\hbox{%
$36\infty2$ (shared)%
}%
}%
\hfill\copy\matricesbox
}%
\ifdim\wd\onelinebox>\myboxwidth
\hbox to \myboxwidth{%
$\Pi_{4,50}$ spans $L_{7.11}$%
\hfil
$36\infty2$ (shared)%
}%
\box\matricesbox
\else
\hbox to \myboxwidth{%
\unhbox\onelinebox
}%
\fi
\else
\hbox to \myboxwidth{%
$\Pi_{4,50}$ spans $L_{7.11}$%
\hfil}%
\hbox to \myboxwidth{%
$36\infty2$ (shared)%
\hfil}%
\box\matricesbox
\fi
}%
\hfill\discretionary{}{}{}%
\setbox\matricesbox=\hbox{%
{$\left[\!\llap{\phantom{%
\begingroup \smaller\smaller\smaller\begin{tabular}{@{}c@{}}%
\phantom{0}\\\phantom{0}\\\phantom{0}
\end{tabular}\endgroup%
}}\right.$}%
\begingroup \smaller\smaller\smaller\begin{tabular}{@{}c@{}}%
-1/5\\\phantom{0}\\\phantom{0}
\end{tabular}\endgroup%
\kern3pt%
\begingroup \smaller\smaller\smaller\begin{tabular}{@{}c@{}}%
\phantom{0}\\4/5\\-2/5
\end{tabular}\endgroup%
\kern3pt%
\begingroup \smaller\smaller\smaller\begin{tabular}{@{}c@{}}%
\phantom{0}\\-2/5\\6/5
\end{tabular}\endgroup%
{$\left.\llap{\phantom{%
\begingroup \smaller\smaller\smaller\begin{tabular}{@{}c@{}}%
\phantom{0}\\\phantom{0}\\\phantom{0}
\end{tabular}\endgroup%
}}\!\right]$}%
{$\left[\!\llap{\phantom{%
\begingroup \smaller\smaller\smaller\begin{tabular}{@{}c@{}}%
0\\0\\0
\end{tabular}\endgroup%
}}\right.$}%
\begingroup \smaller\smaller\smaller\begin{tabular}{@{}c@{}}%
2\\-1\\1
\end{tabular}\endgroup%
\kern3pt%
\begingroup \smaller\smaller\smaller\begin{tabular}{@{}c@{}}%
4\\3\\2
\end{tabular}\endgroup%
\kern3pt%
\begingroup \smaller\smaller\smaller\begin{tabular}{@{}c@{}}%
1\\0\\-1
\end{tabular}\endgroup%
\kern3pt%
\begingroup \smaller\smaller\smaller\begin{tabular}{@{}c@{}}%
2\\-2\\-1
\end{tabular}\endgroup%
{$\left.\llap{\phantom{%
\begingroup \smaller\smaller\smaller\begin{tabular}{@{}c@{}}%
0\\0\\0
\end{tabular}\endgroup%
}}\!\right]$}%
}%
\ifdim\wd\matricesbox>\halfwidth\myboxwidth=\hsize\else\myboxwidth=\halfwidth\fi
\vbox{%
\ifdim\myboxwidth=\hsize
\setbox\onelinebox=\hbox{%
\vbox{\hbox{%
$\Pi_{4,51}$ spans $L_{1.7}$%
}\hbox{%
$4\infty22$ (shared)%
}%
}%
\hfill\copy\matricesbox
}%
\ifdim\wd\onelinebox>\myboxwidth
\hbox to \myboxwidth{%
$\Pi_{4,51}$ spans $L_{1.7}$%
\hfil
$4\infty22$ (shared)%
}%
\box\matricesbox
\else
\hbox to \myboxwidth{%
\unhbox\onelinebox
}%
\fi
\else
\hbox to \myboxwidth{%
$\Pi_{4,51}$ spans $L_{1.7}$%
\hfil}%
\hbox to \myboxwidth{%
$4\infty22$ (shared)%
\hfil}%
\box\matricesbox
\fi
}%
\hfill\discretionary{}{}{}%
\setbox\matricesbox=\hbox{%
{$\left[\!\llap{\phantom{%
\begingroup \smaller\smaller\smaller\begin{tabular}{@{}c@{}}%
\phantom{0}\\\phantom{0}\\\phantom{0}
\end{tabular}\endgroup%
}}\right.$}%
\begingroup \smaller\smaller\smaller\begin{tabular}{@{}c@{}}%
-3/49\\\phantom{0}\\\phantom{0}
\end{tabular}\endgroup%
\kern3pt%
\begingroup \smaller\smaller\smaller\begin{tabular}{@{}c@{}}%
\phantom{0}\\20/49\\-4/49
\end{tabular}\endgroup%
\kern3pt%
\begingroup \smaller\smaller\smaller\begin{tabular}{@{}c@{}}%
\phantom{0}\\-4/49\\40/49
\end{tabular}\endgroup%
{$\left.\llap{\phantom{%
\begingroup \smaller\smaller\smaller\begin{tabular}{@{}c@{}}%
\phantom{0}\\\phantom{0}\\\phantom{0}
\end{tabular}\endgroup%
}}\!\right]$}%
{$\left[\!\llap{\phantom{%
\begingroup \smaller\smaller\smaller\begin{tabular}{@{}c@{}}%
0\\0\\0
\end{tabular}\endgroup%
}}\right.$}%
\begingroup \smaller\smaller\smaller\begin{tabular}{@{}c@{}}%
4\\3\\-1
\end{tabular}\endgroup%
\kern3pt%
\begingroup \smaller\smaller\smaller\begin{tabular}{@{}c@{}}%
8\\-4\\-3
\end{tabular}\endgroup%
\kern3pt%
\begingroup \smaller\smaller\smaller\begin{tabular}{@{}c@{}}%
16\\-6\\4
\end{tabular}\endgroup%
\kern3pt%
\begingroup \smaller\smaller\smaller\begin{tabular}{@{}c@{}}%
1\\1\\1
\end{tabular}\endgroup%
{$\left.\llap{\phantom{%
\begingroup \smaller\smaller\smaller\begin{tabular}{@{}c@{}}%
0\\0\\0
\end{tabular}\endgroup%
}}\!\right]$}%
}%
\ifdim\wd\matricesbox>\halfwidth\myboxwidth=\hsize\else\myboxwidth=\halfwidth\fi
\vbox{%
\ifdim\myboxwidth=\hsize
\setbox\onelinebox=\hbox{%
\vbox{\hbox{%
$\Pi_{4,52}$ spans $L_{123.9}$%
}\hbox{%
$4422$ (shared)%
}%
}%
\hfill\copy\matricesbox
}%
\ifdim\wd\onelinebox>\myboxwidth
\hbox to \myboxwidth{%
$\Pi_{4,52}$ spans $L_{123.9}$%
\hfil
$4422$ (shared)%
}%
\box\matricesbox
\else
\hbox to \myboxwidth{%
\unhbox\onelinebox
}%
\fi
\else
\hbox to \myboxwidth{%
$\Pi_{4,52}$ spans $L_{123.9}$%
\hfil}%
\hbox to \myboxwidth{%
$4422$ (shared)%
\hfil}%
\box\matricesbox
\fi
}%
\hfill\discretionary{}{}{}%
\setbox\matricesbox=\hbox{%
{$\left[\!\llap{\phantom{%
\begingroup \smaller\smaller\smaller\begin{tabular}{@{}c@{}}%
\phantom{0}\\\phantom{0}\\\phantom{0}
\end{tabular}\endgroup%
}}\right.$}%
\begingroup \smaller\smaller\smaller\begin{tabular}{@{}c@{}}%
-1/8\\\phantom{0}\\\phantom{0}
\end{tabular}\endgroup%
\kern3pt%
\begingroup \smaller\smaller\smaller\begin{tabular}{@{}c@{}}%
\phantom{0}\\9/8\\-3/8
\end{tabular}\endgroup%
\kern3pt%
\begingroup \smaller\smaller\smaller\begin{tabular}{@{}c@{}}%
\phantom{0}\\-3/8\\17/8
\end{tabular}\endgroup%
{$\left.\llap{\phantom{%
\begingroup \smaller\smaller\smaller\begin{tabular}{@{}c@{}}%
\phantom{0}\\\phantom{0}\\\phantom{0}
\end{tabular}\endgroup%
}}\!\right]$}%
{$\left[\!\llap{\phantom{%
\begingroup \smaller\smaller\smaller\begin{tabular}{@{}c@{}}%
0\\0\\0
\end{tabular}\endgroup%
}}\right.$}%
\begingroup \smaller\smaller\smaller\begin{tabular}{@{}c@{}}%
1\\-1\\0
\end{tabular}\endgroup%
\kern3pt%
\begingroup \smaller\smaller\smaller\begin{tabular}{@{}c@{}}%
2\\1\\1
\end{tabular}\endgroup%
\kern3pt%
\begingroup \smaller\smaller\smaller\begin{tabular}{@{}c@{}}%
8\\2\\-2
\end{tabular}\endgroup%
\kern3pt%
\begingroup \smaller\smaller\smaller\begin{tabular}{@{}c@{}}%
9\\-2\\-3
\end{tabular}\endgroup%
{$\left.\llap{\phantom{%
\begingroup \smaller\smaller\smaller\begin{tabular}{@{}c@{}}%
0\\0\\0
\end{tabular}\endgroup%
}}\!\right]$}%
}%
\ifdim\wd\matricesbox>\halfwidth\myboxwidth=\hsize\else\myboxwidth=\halfwidth\fi
\vbox{%
\ifdim\myboxwidth=\hsize
\setbox\onelinebox=\hbox{%
\vbox{\hbox{%
$\Pi_{4,53}$ spans $L_{142.3}$%
}\hbox{%
$4\infty22$ (shared)%
}%
}%
\hfill\copy\matricesbox
}%
\ifdim\wd\onelinebox>\myboxwidth
\hbox to \myboxwidth{%
$\Pi_{4,53}$ spans $L_{142.3}$%
\hfil
$4\infty22$ (shared)%
}%
\box\matricesbox
\else
\hbox to \myboxwidth{%
\unhbox\onelinebox
}%
\fi
\else
\hbox to \myboxwidth{%
$\Pi_{4,53}$ spans $L_{142.3}$%
\hfil}%
\hbox to \myboxwidth{%
$4\infty22$ (shared)%
\hfil}%
\box\matricesbox
\fi
}%
\hfill\discretionary{}{}{}%
\setbox\matricesbox=\hbox{%
{$\left[\!\llap{\phantom{%
\begingroup \smaller\smaller\smaller\begin{tabular}{@{}c@{}}%
\phantom{0}\\\phantom{0}\\\phantom{0}
\end{tabular}\endgroup%
}}\right.$}%
\begingroup \smaller\smaller\smaller\begin{tabular}{@{}c@{}}%
-1/9\\\phantom{0}\\\phantom{0}
\end{tabular}\endgroup%
\kern3pt%
\begingroup \smaller\smaller\smaller\begin{tabular}{@{}c@{}}%
\phantom{0}\\4/9\\-2/9
\end{tabular}\endgroup%
\kern3pt%
\begingroup \smaller\smaller\smaller\begin{tabular}{@{}c@{}}%
\phantom{0}\\-2/9\\28/9
\end{tabular}\endgroup%
{$\left.\llap{\phantom{%
\begingroup \smaller\smaller\smaller\begin{tabular}{@{}c@{}}%
\phantom{0}\\\phantom{0}\\\phantom{0}
\end{tabular}\endgroup%
}}\!\right]$}%
{$\left[\!\llap{\phantom{%
\begingroup \smaller\smaller\smaller\begin{tabular}{@{}c@{}}%
0\\0\\0
\end{tabular}\endgroup%
}}\right.$}%
\begingroup \smaller\smaller\smaller\begin{tabular}{@{}c@{}}%
4\\-2\\1
\end{tabular}\endgroup%
\kern3pt%
\begingroup \smaller\smaller\smaller\begin{tabular}{@{}c@{}}%
12\\7\\2
\end{tabular}\endgroup%
\kern3pt%
\begingroup \smaller\smaller\smaller\begin{tabular}{@{}c@{}}%
3\\1\\-1
\end{tabular}\endgroup%
\kern3pt%
\begingroup \smaller\smaller\smaller\begin{tabular}{@{}c@{}}%
12\\-8\\-1
\end{tabular}\endgroup%
{$\left.\llap{\phantom{%
\begingroup \smaller\smaller\smaller\begin{tabular}{@{}c@{}}%
0\\0\\0
\end{tabular}\endgroup%
}}\!\right]$}%
}%
\ifdim\wd\matricesbox>\halfwidth\myboxwidth=\hsize\else\myboxwidth=\halfwidth\fi
\vbox{%
\ifdim\myboxwidth=\hsize
\setbox\onelinebox=\hbox{%
\vbox{\hbox{%
$\Pi_{4,54}$ spans $L_{7.11}$%
}\hbox{%
$6\infty\infty2$%
}%
}%
\hfill\copy\matricesbox
}%
\ifdim\wd\onelinebox>\myboxwidth
\hbox to \myboxwidth{%
$\Pi_{4,54}$ spans $L_{7.11}$%
\hfil
$6\infty\infty2$%
}%
\box\matricesbox
\else
\hbox to \myboxwidth{%
\unhbox\onelinebox
}%
\fi
\else
\hbox to \myboxwidth{%
$\Pi_{4,54}$ spans $L_{7.11}$%
\hfil}%
\hbox to \myboxwidth{%
$6\infty\infty2$%
\hfil}%
\box\matricesbox
\fi
}%
\hfill\discretionary{}{}{}%
\setbox\matricesbox=\hbox{%
{$\left[\!\llap{\phantom{%
\begingroup \smaller\smaller\smaller\begin{tabular}{@{}c@{}}%
\phantom{0}\\\phantom{0}\\\phantom{0}
\end{tabular}\endgroup%
}}\right.$}%
\begingroup \smaller\smaller\smaller\begin{tabular}{@{}c@{}}%
-1/2\\\phantom{0}\\\phantom{0}
\end{tabular}\endgroup%
\kern3pt%
\begingroup \smaller\smaller\smaller\begin{tabular}{@{}c@{}}%
\phantom{0}\\3/2\\-1/2
\end{tabular}\endgroup%
\kern3pt%
\begingroup \smaller\smaller\smaller\begin{tabular}{@{}c@{}}%
\phantom{0}\\-1/2\\3/2
\end{tabular}\endgroup%
{$\left.\llap{\phantom{%
\begingroup \smaller\smaller\smaller\begin{tabular}{@{}c@{}}%
\phantom{0}\\\phantom{0}\\\phantom{0}
\end{tabular}\endgroup%
}}\!\right]$}%
{$\left[\!\llap{\phantom{%
\begingroup \smaller\smaller\smaller\begin{tabular}{@{}c@{}}%
0\\0\\0
\end{tabular}\endgroup%
}}\right.$}%
\begingroup \smaller\smaller\smaller\begin{tabular}{@{}c@{}}%
1\\-1\\0
\end{tabular}\endgroup%
\kern3pt%
\begingroup \smaller\smaller\smaller\begin{tabular}{@{}c@{}}%
4\\1\\3
\end{tabular}\endgroup%
\kern3pt%
\begingroup \smaller\smaller\smaller\begin{tabular}{@{}c@{}}%
1\\1\\0
\end{tabular}\endgroup%
\kern3pt%
\begingroup \smaller\smaller\smaller\begin{tabular}{@{}c@{}}%
1\\0\\-1
\end{tabular}\endgroup%
{$\left.\llap{\phantom{%
\begingroup \smaller\smaller\smaller\begin{tabular}{@{}c@{}}%
0\\0\\0
\end{tabular}\endgroup%
}}\!\right]$}%
}%
\ifdim\wd\matricesbox>\halfwidth\myboxwidth=\hsize\else\myboxwidth=\halfwidth\fi
\vbox{%
\ifdim\myboxwidth=\hsize
\setbox\onelinebox=\hbox{%
\vbox{\hbox{%
$\Pi_{4,55}=\hbox{GN}_{28}$ spans $L_{1.6}$%
}\hbox{%
$\infty\infty2\infty$ (shared)%
}%
}%
\hfill\copy\matricesbox
}%
\ifdim\wd\onelinebox>\myboxwidth
\hbox to \myboxwidth{%
$\Pi_{4,55}=\hbox{GN}_{28}$ spans $L_{1.6}$%
\hfil
$\infty\infty2\infty$ (shared)%
}%
\box\matricesbox
\else
\hbox to \myboxwidth{%
\unhbox\onelinebox
}%
\fi
\else
\hbox to \myboxwidth{%
$\Pi_{4,55}=\hbox{GN}_{28}$ spans $L_{1.6}$%
\hfil}%
\hbox to \myboxwidth{%
$\infty\infty2\infty$ (shared)%
\hfil}%
\box\matricesbox
\fi
}%
\hfill\discretionary{}{}{}%
\setbox\matricesbox=\hbox{%
{$\left[\!\llap{\phantom{%
\begingroup \smaller\smaller\smaller\begin{tabular}{@{}c@{}}%
\phantom{0}\\\phantom{0}\\\phantom{0}
\end{tabular}\endgroup%
}}\right.$}%
\begingroup \smaller\smaller\smaller\begin{tabular}{@{}c@{}}%
-1/9\\\phantom{0}\\\phantom{0}
\end{tabular}\endgroup%
\kern3pt%
\begingroup \smaller\smaller\smaller\begin{tabular}{@{}c@{}}%
\phantom{0}\\4/9\\-2/9
\end{tabular}\endgroup%
\kern3pt%
\begingroup \smaller\smaller\smaller\begin{tabular}{@{}c@{}}%
\phantom{0}\\-2/9\\10/9
\end{tabular}\endgroup%
{$\left.\llap{\phantom{%
\begingroup \smaller\smaller\smaller\begin{tabular}{@{}c@{}}%
\phantom{0}\\\phantom{0}\\\phantom{0}
\end{tabular}\endgroup%
}}\!\right]$}%
{$\left[\!\llap{\phantom{%
\begingroup \smaller\smaller\smaller\begin{tabular}{@{}c@{}}%
0\\0\\0
\end{tabular}\endgroup%
}}\right.$}%
\begingroup \smaller\smaller\smaller\begin{tabular}{@{}c@{}}%
8\\6\\2
\end{tabular}\endgroup%
\kern3pt%
\begingroup \smaller\smaller\smaller\begin{tabular}{@{}c@{}}%
4\\-1\\2
\end{tabular}\endgroup%
\kern3pt%
\begingroup \smaller\smaller\smaller\begin{tabular}{@{}c@{}}%
1\\-1\\-1
\end{tabular}\endgroup%
\kern3pt%
\begingroup \smaller\smaller\smaller\begin{tabular}{@{}c@{}}%
8\\4\\-2
\end{tabular}\endgroup%
{$\left.\llap{\phantom{%
\begingroup \smaller\smaller\smaller\begin{tabular}{@{}c@{}}%
0\\0\\0
\end{tabular}\endgroup%
}}\!\right]$}%
}%
\ifdim\wd\matricesbox>\halfwidth\myboxwidth=\hsize\else\myboxwidth=\halfwidth\fi
\vbox{%
\ifdim\myboxwidth=\hsize
\setbox\onelinebox=\hbox{%
\vbox{\hbox{%
$\Pi_{4,56}$ spans $L_{1.9}$%
}\hbox{%
$4\infty22$ (shared)%
}%
}%
\hfill\copy\matricesbox
}%
\ifdim\wd\onelinebox>\myboxwidth
\hbox to \myboxwidth{%
$\Pi_{4,56}$ spans $L_{1.9}$%
\hfil
$4\infty22$ (shared)%
}%
\box\matricesbox
\else
\hbox to \myboxwidth{%
\unhbox\onelinebox
}%
\fi
\else
\hbox to \myboxwidth{%
$\Pi_{4,56}$ spans $L_{1.9}$%
\hfil}%
\hbox to \myboxwidth{%
$4\infty22$ (shared)%
\hfil}%
\box\matricesbox
\fi
}%
\hfill\discretionary{}{}{}%
\setbox\matricesbox=\hbox{%
{$\left[\!\llap{\phantom{%
\begingroup \smaller\smaller\smaller\begin{tabular}{@{}c@{}}%
\phantom{0}\\\phantom{0}\\\phantom{0}
\end{tabular}\endgroup%
}}\right.$}%
\begingroup \smaller\smaller\smaller\begin{tabular}{@{}c@{}}%
-1/12\\\phantom{0}\\\phantom{0}
\end{tabular}\endgroup%
\kern3pt%
\begingroup \smaller\smaller\smaller\begin{tabular}{@{}c@{}}%
\phantom{0}\\13/12\\-5/12
\end{tabular}\endgroup%
\kern3pt%
\begingroup \smaller\smaller\smaller\begin{tabular}{@{}c@{}}%
\phantom{0}\\-5/12\\13/12
\end{tabular}\endgroup%
{$\left.\llap{\phantom{%
\begingroup \smaller\smaller\smaller\begin{tabular}{@{}c@{}}%
\phantom{0}\\\phantom{0}\\\phantom{0}
\end{tabular}\endgroup%
}}\!\right]$}%
{$\left[\!\llap{\phantom{%
\begingroup \smaller\smaller\smaller\begin{tabular}{@{}c@{}}%
0\\0\\0
\end{tabular}\endgroup%
}}\right.$}%
\begingroup \smaller\smaller\smaller\begin{tabular}{@{}c@{}}%
3\\-2\\-1
\end{tabular}\endgroup%
\kern3pt%
\begingroup \smaller\smaller\smaller\begin{tabular}{@{}c@{}}%
12\\1\\5
\end{tabular}\endgroup%
\kern3pt%
\begingroup \smaller\smaller\smaller\begin{tabular}{@{}c@{}}%
12\\5\\1
\end{tabular}\endgroup%
\kern3pt%
\begingroup \smaller\smaller\smaller\begin{tabular}{@{}c@{}}%
1\\0\\-1
\end{tabular}\endgroup%
{$\left.\llap{\phantom{%
\begingroup \smaller\smaller\smaller\begin{tabular}{@{}c@{}}%
0\\0\\0
\end{tabular}\endgroup%
}}\!\right]$}%
}%
\ifdim\wd\matricesbox>\halfwidth\myboxwidth=\hsize\else\myboxwidth=\halfwidth\fi
\vbox{%
\ifdim\myboxwidth=\hsize
\setbox\onelinebox=\hbox{%
\vbox{\hbox{%
$\Pi_{4,57}$ spans $L_{144.8}$%
}\hbox{%
$\infty\infty22$ (shared)%
}%
}%
\hfill\copy\matricesbox
}%
\ifdim\wd\onelinebox>\myboxwidth
\hbox to \myboxwidth{%
$\Pi_{4,57}$ spans $L_{144.8}$%
\hfil
$\infty\infty22$ (shared)%
}%
\box\matricesbox
\else
\hbox to \myboxwidth{%
\unhbox\onelinebox
}%
\fi
\else
\hbox to \myboxwidth{%
$\Pi_{4,57}$ spans $L_{144.8}$%
\hfil}%
\hbox to \myboxwidth{%
$\infty\infty22$ (shared)%
\hfil}%
\box\matricesbox
\fi
}%
\hfill\discretionary{}{}{}%
\setbox\matricesbox=\hbox{%
{$\left[\!\llap{\phantom{%
\begingroup \smaller\smaller\smaller\begin{tabular}{@{}c@{}}%
\phantom{0}\\\phantom{0}\\\phantom{0}
\end{tabular}\endgroup%
}}\right.$}%
\begingroup \smaller\smaller\smaller\begin{tabular}{@{}c@{}}%
-1/4\\\phantom{0}\\\phantom{0}
\end{tabular}\endgroup%
\kern3pt%
\begingroup \smaller\smaller\smaller\begin{tabular}{@{}c@{}}%
\phantom{0}\\1\\-1/2
\end{tabular}\endgroup%
\kern3pt%
\begingroup \smaller\smaller\smaller\begin{tabular}{@{}c@{}}%
\phantom{0}\\-1/2\\5/4
\end{tabular}\endgroup%
{$\left.\llap{\phantom{%
\begingroup \smaller\smaller\smaller\begin{tabular}{@{}c@{}}%
\phantom{0}\\\phantom{0}\\\phantom{0}
\end{tabular}\endgroup%
}}\!\right]$}%
{$\left[\!\llap{\phantom{%
\begingroup \smaller\smaller\smaller\begin{tabular}{@{}c@{}}%
0\\0\\0
\end{tabular}\endgroup%
}}\right.$}%
\begingroup \smaller\smaller\smaller\begin{tabular}{@{}c@{}}%
4\\3\\2
\end{tabular}\endgroup%
\kern3pt%
\begingroup \smaller\smaller\smaller\begin{tabular}{@{}c@{}}%
4\\-1\\2
\end{tabular}\endgroup%
\kern3pt%
\begingroup \smaller\smaller\smaller\begin{tabular}{@{}c@{}}%
1\\-1\\-1
\end{tabular}\endgroup%
\kern3pt%
\begingroup \smaller\smaller\smaller\begin{tabular}{@{}c@{}}%
4\\1\\-2
\end{tabular}\endgroup%
{$\left.\llap{\phantom{%
\begingroup \smaller\smaller\smaller\begin{tabular}{@{}c@{}}%
0\\0\\0
\end{tabular}\endgroup%
}}\!\right]$}%
}%
\ifdim\wd\matricesbox>\halfwidth\myboxwidth=\hsize\else\myboxwidth=\halfwidth\fi
\vbox{%
\ifdim\myboxwidth=\hsize
\setbox\onelinebox=\hbox{%
\vbox{\hbox{%
$\Pi_{4,58}=\hbox{GN}_{24}$ spans $L_{140.3}$%
}\hbox{%
$\infty\infty2\infty$ (shared)%
}%
}%
\hfill\copy\matricesbox
}%
\ifdim\wd\onelinebox>\myboxwidth
\hbox to \myboxwidth{%
$\Pi_{4,58}=\hbox{GN}_{24}$ spans $L_{140.3}$%
\hfil
$\infty\infty2\infty$ (shared)%
}%
\box\matricesbox
\else
\hbox to \myboxwidth{%
\unhbox\onelinebox
}%
\fi
\else
\hbox to \myboxwidth{%
$\Pi_{4,58}=\hbox{GN}_{24}$ spans $L_{140.3}$%
\hfil}%
\hbox to \myboxwidth{%
$\infty\infty2\infty$ (shared)%
\hfil}%
\box\matricesbox
\fi
}%
\hfill\discretionary{}{}{}%
\setbox\matricesbox=\hbox{%
{$\left[\!\llap{\phantom{%
\begingroup \smaller\smaller\smaller\begin{tabular}{@{}c@{}}%
\phantom{0}\\\phantom{0}\\\phantom{0}
\end{tabular}\endgroup%
}}\right.$}%
\begingroup \smaller\smaller\smaller\begin{tabular}{@{}c@{}}%
-1/8\\\phantom{0}\\\phantom{0}
\end{tabular}\endgroup%
\kern3pt%
\begingroup \smaller\smaller\smaller\begin{tabular}{@{}c@{}}%
\phantom{0}\\1/2\\-1/4
\end{tabular}\endgroup%
\kern3pt%
\begingroup \smaller\smaller\smaller\begin{tabular}{@{}c@{}}%
\phantom{0}\\-1/4\\9/8
\end{tabular}\endgroup%
{$\left.\llap{\phantom{%
\begingroup \smaller\smaller\smaller\begin{tabular}{@{}c@{}}%
\phantom{0}\\\phantom{0}\\\phantom{0}
\end{tabular}\endgroup%
}}\!\right]$}%
{$\left[\!\llap{\phantom{%
\begingroup \smaller\smaller\smaller\begin{tabular}{@{}c@{}}%
0\\0\\0
\end{tabular}\endgroup%
}}\right.$}%
\begingroup \smaller\smaller\smaller\begin{tabular}{@{}c@{}}%
16\\6\\-4
\end{tabular}\endgroup%
\kern3pt%
\begingroup \smaller\smaller\smaller\begin{tabular}{@{}c@{}}%
4\\-3\\-2
\end{tabular}\endgroup%
\kern3pt%
\begingroup \smaller\smaller\smaller\begin{tabular}{@{}c@{}}%
1\\0\\1
\end{tabular}\endgroup%
\kern3pt%
\begingroup \smaller\smaller\smaller\begin{tabular}{@{}c@{}}%
16\\10\\4
\end{tabular}\endgroup%
{$\left.\llap{\phantom{%
\begingroup \smaller\smaller\smaller\begin{tabular}{@{}c@{}}%
0\\0\\0
\end{tabular}\endgroup%
}}\!\right]$}%
}%
\ifdim\wd\matricesbox>\halfwidth\myboxwidth=\hsize\else\myboxwidth=\halfwidth\fi
\vbox{%
\ifdim\myboxwidth=\hsize
\setbox\onelinebox=\hbox{%
\vbox{\hbox{%
$\Pi_{4,59}=\hbox{GN}_{16}$ spans $L_{140.4}$%
}\hbox{%
$\infty\infty2\infty$ (shared)%
}%
}%
\hfill\copy\matricesbox
}%
\ifdim\wd\onelinebox>\myboxwidth
\hbox to \myboxwidth{%
$\Pi_{4,59}=\hbox{GN}_{16}$ spans $L_{140.4}$%
\hfil
$\infty\infty2\infty$ (shared)%
}%
\box\matricesbox
\else
\hbox to \myboxwidth{%
\unhbox\onelinebox
}%
\fi
\else
\hbox to \myboxwidth{%
$\Pi_{4,59}=\hbox{GN}_{16}$ spans $L_{140.4}$%
\hfil}%
\hbox to \myboxwidth{%
$\infty\infty2\infty$ (shared)%
\hfil}%
\box\matricesbox
\fi
}%
\hfill\discretionary{}{}{}%
\setbox\matricesbox=\hbox{%
{$\left[\!\llap{\phantom{%
\begingroup \smaller\smaller\smaller\begin{tabular}{@{}c@{}}%
\phantom{0}\\\phantom{0}\\\phantom{0}
\end{tabular}\endgroup%
}}\right.$}%
\begingroup \smaller\smaller\smaller\begin{tabular}{@{}c@{}}%
-1/8\\\phantom{0}\\\phantom{0}
\end{tabular}\endgroup%
\kern3pt%
\begingroup \smaller\smaller\smaller\begin{tabular}{@{}c@{}}%
\phantom{0}\\9/8\\-3/8
\end{tabular}\endgroup%
\kern3pt%
\begingroup \smaller\smaller\smaller\begin{tabular}{@{}c@{}}%
\phantom{0}\\-3/8\\33/8
\end{tabular}\endgroup%
{$\left.\llap{\phantom{%
\begingroup \smaller\smaller\smaller\begin{tabular}{@{}c@{}}%
\phantom{0}\\\phantom{0}\\\phantom{0}
\end{tabular}\endgroup%
}}\!\right]$}%
{$\left[\!\llap{\phantom{%
\begingroup \smaller\smaller\smaller\begin{tabular}{@{}c@{}}%
0\\0\\0
\end{tabular}\endgroup%
}}\right.$}%
\begingroup \smaller\smaller\smaller\begin{tabular}{@{}c@{}}%
1\\1\\0
\end{tabular}\endgroup%
\kern3pt%
\begingroup \smaller\smaller\smaller\begin{tabular}{@{}c@{}}%
4\\-1\\1
\end{tabular}\endgroup%
\kern3pt%
\begingroup \smaller\smaller\smaller\begin{tabular}{@{}c@{}}%
16\\-6\\-2
\end{tabular}\endgroup%
\kern3pt%
\begingroup \smaller\smaller\smaller\begin{tabular}{@{}c@{}}%
3\\0\\-1
\end{tabular}\endgroup%
{$\left.\llap{\phantom{%
\begingroup \smaller\smaller\smaller\begin{tabular}{@{}c@{}}%
0\\0\\0
\end{tabular}\endgroup%
}}\!\right]$}%
}%
\ifdim\wd\matricesbox>\halfwidth\myboxwidth=\hsize\else\myboxwidth=\halfwidth\fi
\vbox{%
\ifdim\myboxwidth=\hsize
\setbox\onelinebox=\hbox{%
\vbox{\hbox{%
$\Pi_{4,60}$ spans $L_{144.5}$%
}\hbox{%
$\infty\infty22$ (shared)%
}%
}%
\hfill\copy\matricesbox
}%
\ifdim\wd\onelinebox>\myboxwidth
\hbox to \myboxwidth{%
$\Pi_{4,60}$ spans $L_{144.5}$%
\hfil
$\infty\infty22$ (shared)%
}%
\box\matricesbox
\else
\hbox to \myboxwidth{%
\unhbox\onelinebox
}%
\fi
\else
\hbox to \myboxwidth{%
$\Pi_{4,60}$ spans $L_{144.5}$%
\hfil}%
\hbox to \myboxwidth{%
$\infty\infty22$ (shared)%
\hfil}%
\box\matricesbox
\fi
}%
\hfill\discretionary{}{}{}%
\setbox\matricesbox=\hbox{%
{$\left[\!\llap{\phantom{%
\begingroup \smaller\smaller\smaller\begin{tabular}{@{}c@{}}%
\phantom{0}\\\phantom{0}\\\phantom{0}
\end{tabular}\endgroup%
}}\right.$}%
\begingroup \smaller\smaller\smaller\begin{tabular}{@{}c@{}}%
-1/9\\\phantom{0}\\\phantom{0}
\end{tabular}\endgroup%
\kern3pt%
\begingroup \smaller\smaller\smaller\begin{tabular}{@{}c@{}}%
\phantom{0}\\4/9\\-2/9
\end{tabular}\endgroup%
\kern3pt%
\begingroup \smaller\smaller\smaller\begin{tabular}{@{}c@{}}%
\phantom{0}\\-2/9\\28/9
\end{tabular}\endgroup%
{$\left.\llap{\phantom{%
\begingroup \smaller\smaller\smaller\begin{tabular}{@{}c@{}}%
\phantom{0}\\\phantom{0}\\\phantom{0}
\end{tabular}\endgroup%
}}\!\right]$}%
{$\left[\!\llap{\phantom{%
\begingroup \smaller\smaller\smaller\begin{tabular}{@{}c@{}}%
0\\0\\0
\end{tabular}\endgroup%
}}\right.$}%
\begingroup \smaller\smaller\smaller\begin{tabular}{@{}c@{}}%
3\\1\\-1
\end{tabular}\endgroup%
\kern3pt%
\begingroup \smaller\smaller\smaller\begin{tabular}{@{}c@{}}%
3\\2\\1
\end{tabular}\endgroup%
\kern3pt%
\begingroup \smaller\smaller\smaller\begin{tabular}{@{}c@{}}%
12\\-7\\1
\end{tabular}\endgroup%
\kern3pt%
\begingroup \smaller\smaller\smaller\begin{tabular}{@{}c@{}}%
4\\-3\\-1
\end{tabular}\endgroup%
{$\left.\llap{\phantom{%
\begingroup \smaller\smaller\smaller\begin{tabular}{@{}c@{}}%
0\\0\\0
\end{tabular}\endgroup%
}}\!\right]$}%
}%
\ifdim\wd\matricesbox>\halfwidth\myboxwidth=\hsize\else\myboxwidth=\halfwidth\fi
\vbox{%
\ifdim\myboxwidth=\hsize
\setbox\onelinebox=\hbox{%
\vbox{\hbox{%
$\Pi_{4,61}$ spans $L_{7.11}$%
}\hbox{%
$\infty\infty22$ (shared)%
}%
}%
\hfill\copy\matricesbox
}%
\ifdim\wd\onelinebox>\myboxwidth
\hbox to \myboxwidth{%
$\Pi_{4,61}$ spans $L_{7.11}$%
\hfil
$\infty\infty22$ (shared)%
}%
\box\matricesbox
\else
\hbox to \myboxwidth{%
\unhbox\onelinebox
}%
\fi
\else
\hbox to \myboxwidth{%
$\Pi_{4,61}$ spans $L_{7.11}$%
\hfil}%
\hbox to \myboxwidth{%
$\infty\infty22$ (shared)%
\hfil}%
\box\matricesbox
\fi
}%
\hfill\discretionary{}{}{}%
\setbox\matricesbox=\hbox{%
{$\left[\!\llap{\phantom{%
\begingroup \smaller\smaller\smaller\begin{tabular}{@{}c@{}}%
\phantom{0}\\\phantom{0}\\\phantom{0}
\end{tabular}\endgroup%
}}\right.$}%
\begingroup \smaller\smaller\smaller\begin{tabular}{@{}c@{}}%
-1/16\\\phantom{0}\\\phantom{0}
\end{tabular}\endgroup%
\kern3pt%
\begingroup \smaller\smaller\smaller\begin{tabular}{@{}c@{}}%
\phantom{0}\\9/16\\-1/4
\end{tabular}\endgroup%
\kern3pt%
\begingroup \smaller\smaller\smaller\begin{tabular}{@{}c@{}}%
\phantom{0}\\-1/4\\1
\end{tabular}\endgroup%
{$\left.\llap{\phantom{%
\begingroup \smaller\smaller\smaller\begin{tabular}{@{}c@{}}%
\phantom{0}\\\phantom{0}\\\phantom{0}
\end{tabular}\endgroup%
}}\!\right]$}%
{$\left[\!\llap{\phantom{%
\begingroup \smaller\smaller\smaller\begin{tabular}{@{}c@{}}%
0\\0\\0
\end{tabular}\endgroup%
}}\right.$}%
\begingroup \smaller\smaller\smaller\begin{tabular}{@{}c@{}}%
8\\-4\\1
\end{tabular}\endgroup%
\kern3pt%
\begingroup \smaller\smaller\smaller\begin{tabular}{@{}c@{}}%
8\\4\\3
\end{tabular}\endgroup%
\kern3pt%
\begingroup \smaller\smaller\smaller\begin{tabular}{@{}c@{}}%
32\\8\\-6
\end{tabular}\endgroup%
\kern3pt%
\begingroup \smaller\smaller\smaller\begin{tabular}{@{}c@{}}%
1\\-1\\-1
\end{tabular}\endgroup%
{$\left.\llap{\phantom{%
\begingroup \smaller\smaller\smaller\begin{tabular}{@{}c@{}}%
0\\0\\0
\end{tabular}\endgroup%
}}\!\right]$}%
}%
\ifdim\wd\matricesbox>\halfwidth\myboxwidth=\hsize\else\myboxwidth=\halfwidth\fi
\vbox{%
\ifdim\myboxwidth=\hsize
\setbox\onelinebox=\hbox{%
\vbox{\hbox{%
$\Pi_{4,62}$ spans $L_{141.10}$%
}\hbox{%
$\infty\infty22$ (shared)%
}%
}%
\hfill\copy\matricesbox
}%
\ifdim\wd\onelinebox>\myboxwidth
\hbox to \myboxwidth{%
$\Pi_{4,62}$ spans $L_{141.10}$%
\hfil
$\infty\infty22$ (shared)%
}%
\box\matricesbox
\else
\hbox to \myboxwidth{%
\unhbox\onelinebox
}%
\fi
\else
\hbox to \myboxwidth{%
$\Pi_{4,62}$ spans $L_{141.10}$%
\hfil}%
\hbox to \myboxwidth{%
$\infty\infty22$ (shared)%
\hfil}%
\box\matricesbox
\fi
}%
\hfill\discretionary{}{}{}%
\setbox\matricesbox=\hbox{%
{$\left[\!\llap{\phantom{%
\begingroup \smaller\smaller\smaller\begin{tabular}{@{}c@{}}%
\phantom{0}\\\phantom{0}\\\phantom{0}
\end{tabular}\endgroup%
}}\right.$}%
\begingroup \smaller\smaller\smaller\begin{tabular}{@{}c@{}}%
-1/68\\\phantom{0}\\\phantom{0}
\end{tabular}\endgroup%
\kern3pt%
\begingroup \smaller\smaller\smaller\begin{tabular}{@{}c@{}}%
\phantom{0}\\15/17\\-5/17
\end{tabular}\endgroup%
\kern3pt%
\begingroup \smaller\smaller\smaller\begin{tabular}{@{}c@{}}%
\phantom{0}\\-5/17\\30/17
\end{tabular}\endgroup%
{$\left.\llap{\phantom{%
\begingroup \smaller\smaller\smaller\begin{tabular}{@{}c@{}}%
\phantom{0}\\\phantom{0}\\\phantom{0}
\end{tabular}\endgroup%
}}\!\right]$}%
{$\left[\!\llap{\phantom{%
\begingroup \smaller\smaller\smaller\begin{tabular}{@{}c@{}}%
0\\0\\0
\end{tabular}\endgroup%
}}\right.$}%
\begingroup \smaller\smaller\smaller\begin{tabular}{@{}c@{}}%
10\\3\\-1
\end{tabular}\endgroup%
\kern3pt%
\begingroup \smaller\smaller\smaller\begin{tabular}{@{}c@{}}%
10\\-3\\-2
\end{tabular}\endgroup%
\kern3pt%
\begingroup \smaller\smaller\smaller\begin{tabular}{@{}c@{}}%
20\\-4\\2
\end{tabular}\endgroup%
\kern3pt%
\begingroup \smaller\smaller\smaller\begin{tabular}{@{}c@{}}%
2\\1\\1
\end{tabular}\endgroup%
{$\left.\llap{\phantom{%
\begingroup \smaller\smaller\smaller\begin{tabular}{@{}c@{}}%
0\\0\\0
\end{tabular}\endgroup%
}}\!\right]$}%
}%
\ifdim\wd\matricesbox>\halfwidth\myboxwidth=\hsize\else\myboxwidth=\halfwidth\fi
\vbox{%
\ifdim\myboxwidth=\hsize
\setbox\onelinebox=\hbox{%
\vbox{\hbox{%
$\Pi_{4,63}$ spans $L_{6.5}$%
}\hbox{%
$3222$%
}%
}%
\hfill\copy\matricesbox
}%
\ifdim\wd\onelinebox>\myboxwidth
\hbox to \myboxwidth{%
$\Pi_{4,63}$ spans $L_{6.5}$%
\hfil
$3222$%
}%
\box\matricesbox
\else
\hbox to \myboxwidth{%
\unhbox\onelinebox
}%
\fi
\else
\hbox to \myboxwidth{%
$\Pi_{4,63}$ spans $L_{6.5}$%
\hfil}%
\hbox to \myboxwidth{%
$3222$%
\hfil}%
\box\matricesbox
\fi
}%
\hfill\discretionary{}{}{}%
\setbox\matricesbox=\hbox{%
{$\left[\!\llap{\phantom{%
\begingroup \smaller\smaller\smaller\begin{tabular}{@{}c@{}}%
\phantom{0}\\\phantom{0}\\\phantom{0}
\end{tabular}\endgroup%
}}\right.$}%
\begingroup \smaller\smaller\smaller\begin{tabular}{@{}c@{}}%
-1/40\\\phantom{0}\\\phantom{0}
\end{tabular}\endgroup%
\kern3pt%
\begingroup \smaller\smaller\smaller\begin{tabular}{@{}c@{}}%
\phantom{0}\\21/10\\-3/10
\end{tabular}\endgroup%
\kern3pt%
\begingroup \smaller\smaller\smaller\begin{tabular}{@{}c@{}}%
\phantom{0}\\-3/10\\29/10
\end{tabular}\endgroup%
{$\left.\llap{\phantom{%
\begingroup \smaller\smaller\smaller\begin{tabular}{@{}c@{}}%
\phantom{0}\\\phantom{0}\\\phantom{0}
\end{tabular}\endgroup%
}}\!\right]$}%
{$\left[\!\llap{\phantom{%
\begingroup \smaller\smaller\smaller\begin{tabular}{@{}c@{}}%
0\\0\\0
\end{tabular}\endgroup%
}}\right.$}%
\begingroup \smaller\smaller\smaller\begin{tabular}{@{}c@{}}%
2\\1\\0
\end{tabular}\endgroup%
\kern3pt%
\begingroup \smaller\smaller\smaller\begin{tabular}{@{}c@{}}%
4\\-1\\-1
\end{tabular}\endgroup%
\kern3pt%
\begingroup \smaller\smaller\smaller\begin{tabular}{@{}c@{}}%
10\\-2\\1
\end{tabular}\endgroup%
\kern3pt%
\begingroup \smaller\smaller\smaller\begin{tabular}{@{}c@{}}%
48\\2\\6
\end{tabular}\endgroup%
{$\left.\llap{\phantom{%
\begingroup \smaller\smaller\smaller\begin{tabular}{@{}c@{}}%
0\\0\\0
\end{tabular}\endgroup%
}}\!\right]$}%
}%
\ifdim\wd\matricesbox>\halfwidth\myboxwidth=\hsize\else\myboxwidth=\halfwidth\fi
\vbox{%
\ifdim\myboxwidth=\hsize
\setbox\onelinebox=\hbox{%
\vbox{\hbox{%
$\Pi_{4,64}$ spans $L_{19.1}$%
}\hbox{%
$4222$ (shared)%
}%
}%
\hfill\copy\matricesbox
}%
\ifdim\wd\onelinebox>\myboxwidth
\hbox to \myboxwidth{%
$\Pi_{4,64}$ spans $L_{19.1}$%
\hfil
$4222$ (shared)%
}%
\box\matricesbox
\else
\hbox to \myboxwidth{%
\unhbox\onelinebox
}%
\fi
\else
\hbox to \myboxwidth{%
$\Pi_{4,64}$ spans $L_{19.1}$%
\hfil}%
\hbox to \myboxwidth{%
$4222$ (shared)%
\hfil}%
\box\matricesbox
\fi
}%
\hfill\discretionary{}{}{}%
\setbox\matricesbox=\hbox{%
{$\left[\!\llap{\phantom{%
\begingroup \smaller\smaller\smaller\begin{tabular}{@{}c@{}}%
\phantom{0}\\\phantom{0}\\\phantom{0}
\end{tabular}\endgroup%
}}\right.$}%
\begingroup \smaller\smaller\smaller\begin{tabular}{@{}c@{}}%
-1/14\\\phantom{0}\\\phantom{0}
\end{tabular}\endgroup%
\kern3pt%
\begingroup \smaller\smaller\smaller\begin{tabular}{@{}c@{}}%
\phantom{0}\\4/7\\-1/7
\end{tabular}\endgroup%
\kern3pt%
\begingroup \smaller\smaller\smaller\begin{tabular}{@{}c@{}}%
\phantom{0}\\-1/7\\11/14
\end{tabular}\endgroup%
{$\left.\llap{\phantom{%
\begingroup \smaller\smaller\smaller\begin{tabular}{@{}c@{}}%
\phantom{0}\\\phantom{0}\\\phantom{0}
\end{tabular}\endgroup%
}}\!\right]$}%
{$\left[\!\llap{\phantom{%
\begingroup \smaller\smaller\smaller\begin{tabular}{@{}c@{}}%
0\\0\\0
\end{tabular}\endgroup%
}}\right.$}%
\begingroup \smaller\smaller\smaller\begin{tabular}{@{}c@{}}%
1\\1\\1
\end{tabular}\endgroup%
\kern3pt%
\begingroup \smaller\smaller\smaller\begin{tabular}{@{}c@{}}%
2\\-2\\0
\end{tabular}\endgroup%
\kern3pt%
\begingroup \smaller\smaller\smaller\begin{tabular}{@{}c@{}}%
8\\-2\\-4
\end{tabular}\endgroup%
\kern3pt%
\begingroup \smaller\smaller\smaller\begin{tabular}{@{}c@{}}%
3\\2\\-1
\end{tabular}\endgroup%
{$\left.\llap{\phantom{%
\begingroup \smaller\smaller\smaller\begin{tabular}{@{}c@{}}%
0\\0\\0
\end{tabular}\endgroup%
}}\!\right]$}%
}%
\ifdim\wd\matricesbox>\halfwidth\myboxwidth=\hsize\else\myboxwidth=\halfwidth\fi
\vbox{%
\ifdim\myboxwidth=\hsize
\setbox\onelinebox=\hbox{%
\vbox{\hbox{%
$\Pi_{4,65}$ spans $L_{123.3}$%
}\hbox{%
$4222$ (shared)%
}%
}%
\hfill\copy\matricesbox
}%
\ifdim\wd\onelinebox>\myboxwidth
\hbox to \myboxwidth{%
$\Pi_{4,65}$ spans $L_{123.3}$%
\hfil
$4222$ (shared)%
}%
\box\matricesbox
\else
\hbox to \myboxwidth{%
\unhbox\onelinebox
}%
\fi
\else
\hbox to \myboxwidth{%
$\Pi_{4,65}$ spans $L_{123.3}$%
\hfil}%
\hbox to \myboxwidth{%
$4222$ (shared)%
\hfil}%
\box\matricesbox
\fi
}%
\hfill\discretionary{}{}{}%
\setbox\matricesbox=\hbox{%
{$\left[\!\llap{\phantom{%
\begingroup \smaller\smaller\smaller\begin{tabular}{@{}c@{}}%
\phantom{0}\\\phantom{0}\\\phantom{0}
\end{tabular}\endgroup%
}}\right.$}%
\begingroup \smaller\smaller\smaller\begin{tabular}{@{}c@{}}%
-1/10\\\phantom{0}\\\phantom{0}
\end{tabular}\endgroup%
\kern3pt%
\begingroup \smaller\smaller\smaller\begin{tabular}{@{}c@{}}%
\phantom{0}\\11/10\\-3/10
\end{tabular}\endgroup%
\kern3pt%
\begingroup \smaller\smaller\smaller\begin{tabular}{@{}c@{}}%
\phantom{0}\\-3/10\\19/10
\end{tabular}\endgroup%
{$\left.\llap{\phantom{%
\begingroup \smaller\smaller\smaller\begin{tabular}{@{}c@{}}%
\phantom{0}\\\phantom{0}\\\phantom{0}
\end{tabular}\endgroup%
}}\!\right]$}%
{$\left[\!\llap{\phantom{%
\begingroup \smaller\smaller\smaller\begin{tabular}{@{}c@{}}%
0\\0\\0
\end{tabular}\endgroup%
}}\right.$}%
\begingroup \smaller\smaller\smaller\begin{tabular}{@{}c@{}}%
1\\-1\\0
\end{tabular}\endgroup%
\kern3pt%
\begingroup \smaller\smaller\smaller\begin{tabular}{@{}c@{}}%
2\\1\\1
\end{tabular}\endgroup%
\kern3pt%
\begingroup \smaller\smaller\smaller\begin{tabular}{@{}c@{}}%
20\\7\\-1
\end{tabular}\endgroup%
\kern3pt%
\begingroup \smaller\smaller\smaller\begin{tabular}{@{}c@{}}%
5\\-1\\-2
\end{tabular}\endgroup%
{$\left.\llap{\phantom{%
\begingroup \smaller\smaller\smaller\begin{tabular}{@{}c@{}}%
0\\0\\0
\end{tabular}\endgroup%
}}\!\right]$}%
}%
\ifdim\wd\matricesbox>\halfwidth\myboxwidth=\hsize\else\myboxwidth=\halfwidth\fi
\vbox{%
\ifdim\myboxwidth=\hsize
\setbox\onelinebox=\hbox{%
\vbox{\hbox{%
$\Pi_{4,66}$ spans $L_{5.6}$%
}\hbox{%
$42\infty2$%
}%
}%
\hfill\copy\matricesbox
}%
\ifdim\wd\onelinebox>\myboxwidth
\hbox to \myboxwidth{%
$\Pi_{4,66}$ spans $L_{5.6}$%
\hfil
$42\infty2$%
}%
\box\matricesbox
\else
\hbox to \myboxwidth{%
\unhbox\onelinebox
}%
\fi
\else
\hbox to \myboxwidth{%
$\Pi_{4,66}$ spans $L_{5.6}$%
\hfil}%
\hbox to \myboxwidth{%
$42\infty2$%
\hfil}%
\box\matricesbox
\fi
}%
\hfill\discretionary{}{}{}%
\setbox\matricesbox=\hbox{%
{$\left[\!\llap{\phantom{%
\begingroup \smaller\smaller\smaller\begin{tabular}{@{}c@{}}%
\phantom{0}\\\phantom{0}\\\phantom{0}
\end{tabular}\endgroup%
}}\right.$}%
\begingroup \smaller\smaller\smaller\begin{tabular}{@{}c@{}}%
-1/20\\\phantom{0}\\\phantom{0}
\end{tabular}\endgroup%
\kern3pt%
\begingroup \smaller\smaller\smaller\begin{tabular}{@{}c@{}}%
\phantom{0}\\21/20\\-9/20
\end{tabular}\endgroup%
\kern3pt%
\begingroup \smaller\smaller\smaller\begin{tabular}{@{}c@{}}%
\phantom{0}\\-9/20\\21/20
\end{tabular}\endgroup%
{$\left.\llap{\phantom{%
\begingroup \smaller\smaller\smaller\begin{tabular}{@{}c@{}}%
\phantom{0}\\\phantom{0}\\\phantom{0}
\end{tabular}\endgroup%
}}\!\right]$}%
{$\left[\!\llap{\phantom{%
\begingroup \smaller\smaller\smaller\begin{tabular}{@{}c@{}}%
0\\0\\0
\end{tabular}\endgroup%
}}\right.$}%
\begingroup \smaller\smaller\smaller\begin{tabular}{@{}c@{}}%
3\\1\\2
\end{tabular}\endgroup%
\kern3pt%
\begingroup \smaller\smaller\smaller\begin{tabular}{@{}c@{}}%
6\\-3\\-1
\end{tabular}\endgroup%
\kern3pt%
\begingroup \smaller\smaller\smaller\begin{tabular}{@{}c@{}}%
6\\-1\\-3
\end{tabular}\endgroup%
\kern3pt%
\begingroup \smaller\smaller\smaller\begin{tabular}{@{}c@{}}%
1\\1\\0
\end{tabular}\endgroup%
{$\left.\llap{\phantom{%
\begingroup \smaller\smaller\smaller\begin{tabular}{@{}c@{}}%
0\\0\\0
\end{tabular}\endgroup%
}}\!\right]$}%
}%
\ifdim\wd\matricesbox>\halfwidth\myboxwidth=\hsize\else\myboxwidth=\halfwidth\fi
\vbox{%
\ifdim\myboxwidth=\hsize
\setbox\onelinebox=\hbox{%
\vbox{\hbox{%
$\Pi_{4,67}$ spans $L_{123.6}$%
}\hbox{%
$4222$ (shared)%
}%
}%
\hfill\copy\matricesbox
}%
\ifdim\wd\onelinebox>\myboxwidth
\hbox to \myboxwidth{%
$\Pi_{4,67}$ spans $L_{123.6}$%
\hfil
$4222$ (shared)%
}%
\box\matricesbox
\else
\hbox to \myboxwidth{%
\unhbox\onelinebox
}%
\fi
\else
\hbox to \myboxwidth{%
$\Pi_{4,67}$ spans $L_{123.6}$%
\hfil}%
\hbox to \myboxwidth{%
$4222$ (shared)%
\hfil}%
\box\matricesbox
\fi
}%
\hfill\discretionary{}{}{}%
\setbox\matricesbox=\hbox{%
{$\left[\!\llap{\phantom{%
\begingroup \smaller\smaller\smaller\begin{tabular}{@{}c@{}}%
\phantom{0}\\\phantom{0}\\\phantom{0}
\end{tabular}\endgroup%
}}\right.$}%
\begingroup \smaller\smaller\smaller\begin{tabular}{@{}c@{}}%
-1/19\\\phantom{0}\\\phantom{0}
\end{tabular}\endgroup%
\kern3pt%
\begingroup \smaller\smaller\smaller\begin{tabular}{@{}c@{}}%
\phantom{0}\\6/19\\-3/19
\end{tabular}\endgroup%
\kern3pt%
\begingroup \smaller\smaller\smaller\begin{tabular}{@{}c@{}}%
\phantom{0}\\-3/19\\30/19
\end{tabular}\endgroup%
{$\left.\llap{\phantom{%
\begingroup \smaller\smaller\smaller\begin{tabular}{@{}c@{}}%
\phantom{0}\\\phantom{0}\\\phantom{0}
\end{tabular}\endgroup%
}}\!\right]$}%
{$\left[\!\llap{\phantom{%
\begingroup \smaller\smaller\smaller\begin{tabular}{@{}c@{}}%
0\\0\\0
\end{tabular}\endgroup%
}}\right.$}%
\begingroup \smaller\smaller\smaller\begin{tabular}{@{}c@{}}%
6\\4\\-1
\end{tabular}\endgroup%
\kern3pt%
\begingroup \smaller\smaller\smaller\begin{tabular}{@{}c@{}}%
3\\-3\\-1
\end{tabular}\endgroup%
\kern3pt%
\begingroup \smaller\smaller\smaller\begin{tabular}{@{}c@{}}%
3\\-2\\1
\end{tabular}\endgroup%
\kern3pt%
\begingroup \smaller\smaller\smaller\begin{tabular}{@{}c@{}}%
2\\2\\1
\end{tabular}\endgroup%
{$\left.\llap{\phantom{%
\begingroup \smaller\smaller\smaller\begin{tabular}{@{}c@{}}%
0\\0\\0
\end{tabular}\endgroup%
}}\!\right]$}%
}%
\ifdim\wd\matricesbox>\halfwidth\myboxwidth=\hsize\else\myboxwidth=\halfwidth\fi
\vbox{%
\ifdim\myboxwidth=\hsize
\setbox\onelinebox=\hbox{%
\vbox{\hbox{%
$\Pi_{4,68}$ spans $L_{3.3}$%
}\hbox{%
$4222$ (shared)%
}%
}%
\hfill\copy\matricesbox
}%
\ifdim\wd\onelinebox>\myboxwidth
\hbox to \myboxwidth{%
$\Pi_{4,68}$ spans $L_{3.3}$%
\hfil
$4222$ (shared)%
}%
\box\matricesbox
\else
\hbox to \myboxwidth{%
\unhbox\onelinebox
}%
\fi
\else
\hbox to \myboxwidth{%
$\Pi_{4,68}$ spans $L_{3.3}$%
\hfil}%
\hbox to \myboxwidth{%
$4222$ (shared)%
\hfil}%
\box\matricesbox
\fi
}%
\hfill\discretionary{}{}{}%
\setbox\matricesbox=\hbox{%
{$\left[\!\llap{\phantom{%
\begingroup \smaller\smaller\smaller\begin{tabular}{@{}c@{}}%
\phantom{0}\\\phantom{0}\\\phantom{0}
\end{tabular}\endgroup%
}}\right.$}%
\begingroup \smaller\smaller\smaller\begin{tabular}{@{}c@{}}%
-1/184\\\phantom{0}\\\phantom{0}
\end{tabular}\endgroup%
\kern3pt%
\begingroup \smaller\smaller\smaller\begin{tabular}{@{}c@{}}%
\phantom{0}\\165/46\\-75/46
\end{tabular}\endgroup%
\kern3pt%
\begingroup \smaller\smaller\smaller\begin{tabular}{@{}c@{}}%
\phantom{0}\\-75/46\\285/46
\end{tabular}\endgroup%
{$\left.\llap{\phantom{%
\begingroup \smaller\smaller\smaller\begin{tabular}{@{}c@{}}%
\phantom{0}\\\phantom{0}\\\phantom{0}
\end{tabular}\endgroup%
}}\!\right]$}%
{$\left[\!\llap{\phantom{%
\begingroup \smaller\smaller\smaller\begin{tabular}{@{}c@{}}%
0\\0\\0
\end{tabular}\endgroup%
}}\right.$}%
\begingroup \smaller\smaller\smaller\begin{tabular}{@{}c@{}}%
30\\1\\-2
\end{tabular}\endgroup%
\kern3pt%
\begingroup \smaller\smaller\smaller\begin{tabular}{@{}c@{}}%
60\\-5\\-1
\end{tabular}\endgroup%
\kern3pt%
\begingroup \smaller\smaller\smaller\begin{tabular}{@{}c@{}}%
6\\0\\1
\end{tabular}\endgroup%
\kern3pt%
\begingroup \smaller\smaller\smaller\begin{tabular}{@{}c@{}}%
80\\6\\2
\end{tabular}\endgroup%
{$\left.\llap{\phantom{%
\begingroup \smaller\smaller\smaller\begin{tabular}{@{}c@{}}%
0\\0\\0
\end{tabular}\endgroup%
}}\!\right]$}%
}%
\ifdim\wd\matricesbox>\halfwidth\myboxwidth=\hsize\else\myboxwidth=\halfwidth\fi
\vbox{%
\ifdim\myboxwidth=\hsize
\setbox\onelinebox=\hbox{%
\vbox{\hbox{%
$\Pi_{4,69}$ spans $L_{19.10}$%
}\hbox{%
$4222$ (shared)%
}%
}%
\hfill\copy\matricesbox
}%
\ifdim\wd\onelinebox>\myboxwidth
\hbox to \myboxwidth{%
$\Pi_{4,69}$ spans $L_{19.10}$%
\hfil
$4222$ (shared)%
}%
\box\matricesbox
\else
\hbox to \myboxwidth{%
\unhbox\onelinebox
}%
\fi
\else
\hbox to \myboxwidth{%
$\Pi_{4,69}$ spans $L_{19.10}$%
\hfil}%
\hbox to \myboxwidth{%
$4222$ (shared)%
\hfil}%
\box\matricesbox
\fi
}%
\hfill\discretionary{}{}{}%
\setbox\matricesbox=\hbox{%
{$\left[\!\llap{\phantom{%
\begingroup \smaller\smaller\smaller\begin{tabular}{@{}c@{}}%
\phantom{0}\\\phantom{0}\\\phantom{0}
\end{tabular}\endgroup%
}}\right.$}%
\begingroup \smaller\smaller\smaller\begin{tabular}{@{}c@{}}%
-1/49\\\phantom{0}\\\phantom{0}
\end{tabular}\endgroup%
\kern3pt%
\begingroup \smaller\smaller\smaller\begin{tabular}{@{}c@{}}%
\phantom{0}\\30/49\\-6/49
\end{tabular}\endgroup%
\kern3pt%
\begingroup \smaller\smaller\smaller\begin{tabular}{@{}c@{}}%
\phantom{0}\\-6/49\\60/49
\end{tabular}\endgroup%
{$\left.\llap{\phantom{%
\begingroup \smaller\smaller\smaller\begin{tabular}{@{}c@{}}%
\phantom{0}\\\phantom{0}\\\phantom{0}
\end{tabular}\endgroup%
}}\!\right]$}%
{$\left[\!\llap{\phantom{%
\begingroup \smaller\smaller\smaller\begin{tabular}{@{}c@{}}%
0\\0\\0
\end{tabular}\endgroup%
}}\right.$}%
\begingroup \smaller\smaller\smaller\begin{tabular}{@{}c@{}}%
24\\6\\4
\end{tabular}\endgroup%
\kern3pt%
\begingroup \smaller\smaller\smaller\begin{tabular}{@{}c@{}}%
12\\2\\-3
\end{tabular}\endgroup%
\kern3pt%
\begingroup \smaller\smaller\smaller\begin{tabular}{@{}c@{}}%
3\\-2\\-1
\end{tabular}\endgroup%
\kern3pt%
\begingroup \smaller\smaller\smaller\begin{tabular}{@{}c@{}}%
2\\-1\\1
\end{tabular}\endgroup%
{$\left.\llap{\phantom{%
\begingroup \smaller\smaller\smaller\begin{tabular}{@{}c@{}}%
0\\0\\0
\end{tabular}\endgroup%
}}\!\right]$}%
}%
\ifdim\wd\matricesbox>\halfwidth\myboxwidth=\hsize\else\myboxwidth=\halfwidth\fi
\vbox{%
\ifdim\myboxwidth=\hsize
\setbox\onelinebox=\hbox{%
\vbox{\hbox{%
$\Pi_{4,70}$ spans $L_{123.8}$%
}\hbox{%
$4222$ (shared)%
}%
}%
\hfill\copy\matricesbox
}%
\ifdim\wd\onelinebox>\myboxwidth
\hbox to \myboxwidth{%
$\Pi_{4,70}$ spans $L_{123.8}$%
\hfil
$4222$ (shared)%
}%
\box\matricesbox
\else
\hbox to \myboxwidth{%
\unhbox\onelinebox
}%
\fi
\else
\hbox to \myboxwidth{%
$\Pi_{4,70}$ spans $L_{123.8}$%
\hfil}%
\hbox to \myboxwidth{%
$4222$ (shared)%
\hfil}%
\box\matricesbox
\fi
}%
\hfill\discretionary{}{}{}%
\setbox\matricesbox=\hbox{%
{$\left[\!\llap{\phantom{%
\begingroup \smaller\smaller\smaller\begin{tabular}{@{}c@{}}%
\phantom{0}\\\phantom{0}\\\phantom{0}
\end{tabular}\endgroup%
}}\right.$}%
\begingroup \smaller\smaller\smaller\begin{tabular}{@{}c@{}}%
-1/28\\\phantom{0}\\\phantom{0}
\end{tabular}\endgroup%
\kern3pt%
\begingroup \smaller\smaller\smaller\begin{tabular}{@{}c@{}}%
\phantom{0}\\15/7\\-3/7
\end{tabular}\endgroup%
\kern3pt%
\begingroup \smaller\smaller\smaller\begin{tabular}{@{}c@{}}%
\phantom{0}\\-3/7\\30/7
\end{tabular}\endgroup%
{$\left.\llap{\phantom{%
\begingroup \smaller\smaller\smaller\begin{tabular}{@{}c@{}}%
\phantom{0}\\\phantom{0}\\\phantom{0}
\end{tabular}\endgroup%
}}\!\right]$}%
{$\left[\!\llap{\phantom{%
\begingroup \smaller\smaller\smaller\begin{tabular}{@{}c@{}}%
0\\0\\0
\end{tabular}\endgroup%
}}\right.$}%
\begingroup \smaller\smaller\smaller\begin{tabular}{@{}c@{}}%
2\\1\\0
\end{tabular}\endgroup%
\kern3pt%
\begingroup \smaller\smaller\smaller\begin{tabular}{@{}c@{}}%
6\\-1\\1
\end{tabular}\endgroup%
\kern3pt%
\begingroup \smaller\smaller\smaller\begin{tabular}{@{}c@{}}%
14\\-3\\-1
\end{tabular}\endgroup%
\kern3pt%
\begingroup \smaller\smaller\smaller\begin{tabular}{@{}c@{}}%
12\\0\\-2
\end{tabular}\endgroup%
{$\left.\llap{\phantom{%
\begingroup \smaller\smaller\smaller\begin{tabular}{@{}c@{}}%
0\\0\\0
\end{tabular}\endgroup%
}}\!\right]$}%
}%
\ifdim\wd\matricesbox>\halfwidth\myboxwidth=\hsize\else\myboxwidth=\halfwidth\fi
\vbox{%
\ifdim\myboxwidth=\hsize
\setbox\onelinebox=\hbox{%
\vbox{\hbox{%
$\Pi_{4,71}$ spans $L_{22.5}$%
}\hbox{%
$6222$ (shared)%
}%
}%
\hfill\copy\matricesbox
}%
\ifdim\wd\onelinebox>\myboxwidth
\hbox to \myboxwidth{%
$\Pi_{4,71}$ spans $L_{22.5}$%
\hfil
$6222$ (shared)%
}%
\box\matricesbox
\else
\hbox to \myboxwidth{%
\unhbox\onelinebox
}%
\fi
\else
\hbox to \myboxwidth{%
$\Pi_{4,71}$ spans $L_{22.5}$%
\hfil}%
\hbox to \myboxwidth{%
$6222$ (shared)%
\hfil}%
\box\matricesbox
\fi
}%
\hfill\discretionary{}{}{}%
\setbox\matricesbox=\hbox{%
{$\left[\!\llap{\phantom{%
\begingroup \smaller\smaller\smaller\begin{tabular}{@{}c@{}}%
\phantom{0}\\\phantom{0}\\\phantom{0}
\end{tabular}\endgroup%
}}\right.$}%
\begingroup \smaller\smaller\smaller\begin{tabular}{@{}c@{}}%
-1/24\\\phantom{0}\\\phantom{0}
\end{tabular}\endgroup%
\kern3pt%
\begingroup \smaller\smaller\smaller\begin{tabular}{@{}c@{}}%
\phantom{0}\\13/6\\-1/3
\end{tabular}\endgroup%
\kern3pt%
\begingroup \smaller\smaller\smaller\begin{tabular}{@{}c@{}}%
\phantom{0}\\-1/3\\14/3
\end{tabular}\endgroup%
{$\left.\llap{\phantom{%
\begingroup \smaller\smaller\smaller\begin{tabular}{@{}c@{}}%
\phantom{0}\\\phantom{0}\\\phantom{0}
\end{tabular}\endgroup%
}}\!\right]$}%
{$\left[\!\llap{\phantom{%
\begingroup \smaller\smaller\smaller\begin{tabular}{@{}c@{}}%
0\\0\\0
\end{tabular}\endgroup%
}}\right.$}%
\begingroup \smaller\smaller\smaller\begin{tabular}{@{}c@{}}%
2\\1\\0
\end{tabular}\endgroup%
\kern3pt%
\begingroup \smaller\smaller\smaller\begin{tabular}{@{}c@{}}%
6\\-1\\1
\end{tabular}\endgroup%
\kern3pt%
\begingroup \smaller\smaller\smaller\begin{tabular}{@{}c@{}}%
20\\-4\\-1
\end{tabular}\endgroup%
\kern3pt%
\begingroup \smaller\smaller\smaller\begin{tabular}{@{}c@{}}%
4\\0\\-1
\end{tabular}\endgroup%
{$\left.\llap{\phantom{%
\begingroup \smaller\smaller\smaller\begin{tabular}{@{}c@{}}%
0\\0\\0
\end{tabular}\endgroup%
}}\!\right]$}%
}%
\ifdim\wd\matricesbox>\halfwidth\myboxwidth=\hsize\else\myboxwidth=\halfwidth\fi
\vbox{%
\ifdim\myboxwidth=\hsize
\setbox\onelinebox=\hbox{%
\vbox{\hbox{%
$\Pi_{4,72}$ spans $L_{16.3}$%
}\hbox{%
$6222$ (shared)%
}%
}%
\hfill\copy\matricesbox
}%
\ifdim\wd\onelinebox>\myboxwidth
\hbox to \myboxwidth{%
$\Pi_{4,72}$ spans $L_{16.3}$%
\hfil
$6222$ (shared)%
}%
\box\matricesbox
\else
\hbox to \myboxwidth{%
\unhbox\onelinebox
}%
\fi
\else
\hbox to \myboxwidth{%
$\Pi_{4,72}$ spans $L_{16.3}$%
\hfil}%
\hbox to \myboxwidth{%
$6222$ (shared)%
\hfil}%
\box\matricesbox
\fi
}%
\hfill\discretionary{}{}{}%
\setbox\matricesbox=\hbox{%
{$\left[\!\llap{\phantom{%
\begingroup \smaller\smaller\smaller\begin{tabular}{@{}c@{}}%
\phantom{0}\\\phantom{0}\\\phantom{0}
\end{tabular}\endgroup%
}}\right.$}%
\begingroup \smaller\smaller\smaller\begin{tabular}{@{}c@{}}%
-1/104\\\phantom{0}\\\phantom{0}
\end{tabular}\endgroup%
\kern3pt%
\begingroup \smaller\smaller\smaller\begin{tabular}{@{}c@{}}%
\phantom{0}\\165/26\\-15/26
\end{tabular}\endgroup%
\kern3pt%
\begingroup \smaller\smaller\smaller\begin{tabular}{@{}c@{}}%
\phantom{0}\\-15/26\\285/26
\end{tabular}\endgroup%
{$\left.\llap{\phantom{%
\begingroup \smaller\smaller\smaller\begin{tabular}{@{}c@{}}%
\phantom{0}\\\phantom{0}\\\phantom{0}
\end{tabular}\endgroup%
}}\!\right]$}%
{$\left[\!\llap{\phantom{%
\begingroup \smaller\smaller\smaller\begin{tabular}{@{}c@{}}%
0\\0\\0
\end{tabular}\endgroup%
}}\right.$}%
\begingroup \smaller\smaller\smaller\begin{tabular}{@{}c@{}}%
10\\0\\1
\end{tabular}\endgroup%
\kern3pt%
\begingroup \smaller\smaller\smaller\begin{tabular}{@{}c@{}}%
30\\2\\-1
\end{tabular}\endgroup%
\kern3pt%
\begingroup \smaller\smaller\smaller\begin{tabular}{@{}c@{}}%
160\\-2\\-6
\end{tabular}\endgroup%
\kern3pt%
\begingroup \smaller\smaller\smaller\begin{tabular}{@{}c@{}}%
6\\-1\\0
\end{tabular}\endgroup%
{$\left.\llap{\phantom{%
\begingroup \smaller\smaller\smaller\begin{tabular}{@{}c@{}}%
0\\0\\0
\end{tabular}\endgroup%
}}\!\right]$}%
}%
\ifdim\wd\matricesbox>\halfwidth\myboxwidth=\hsize\else\myboxwidth=\halfwidth\fi
\vbox{%
\ifdim\myboxwidth=\hsize
\setbox\onelinebox=\hbox{%
\vbox{\hbox{%
$\Pi_{4,73}$ spans $L_{31.13}$%
}\hbox{%
$6222$ (shared)%
}%
}%
\hfill\copy\matricesbox
}%
\ifdim\wd\onelinebox>\myboxwidth
\hbox to \myboxwidth{%
$\Pi_{4,73}$ spans $L_{31.13}$%
\hfil
$6222$ (shared)%
}%
\box\matricesbox
\else
\hbox to \myboxwidth{%
\unhbox\onelinebox
}%
\fi
\else
\hbox to \myboxwidth{%
$\Pi_{4,73}$ spans $L_{31.13}$%
\hfil}%
\hbox to \myboxwidth{%
$6222$ (shared)%
\hfil}%
\box\matricesbox
\fi
}%
\hfill\discretionary{}{}{}%
\setbox\matricesbox=\hbox{%
{$\left[\!\llap{\phantom{%
\begingroup \smaller\smaller\smaller\begin{tabular}{@{}c@{}}%
\phantom{0}\\\phantom{0}\\\phantom{0}
\end{tabular}\endgroup%
}}\right.$}%
\begingroup \smaller\smaller\smaller\begin{tabular}{@{}c@{}}%
-1/92\\\phantom{0}\\\phantom{0}
\end{tabular}\endgroup%
\kern3pt%
\begingroup \smaller\smaller\smaller\begin{tabular}{@{}c@{}}%
\phantom{0}\\35/23\\-7/23
\end{tabular}\endgroup%
\kern3pt%
\begingroup \smaller\smaller\smaller\begin{tabular}{@{}c@{}}%
\phantom{0}\\-7/23\\98/23
\end{tabular}\endgroup%
{$\left.\llap{\phantom{%
\begingroup \smaller\smaller\smaller\begin{tabular}{@{}c@{}}%
\phantom{0}\\\phantom{0}\\\phantom{0}
\end{tabular}\endgroup%
}}\!\right]$}%
{$\left[\!\llap{\phantom{%
\begingroup \smaller\smaller\smaller\begin{tabular}{@{}c@{}}%
0\\0\\0
\end{tabular}\endgroup%
}}\right.$}%
\begingroup \smaller\smaller\smaller\begin{tabular}{@{}c@{}}%
42\\-5\\2
\end{tabular}\endgroup%
\kern3pt%
\begingroup \smaller\smaller\smaller\begin{tabular}{@{}c@{}}%
14\\3\\1
\end{tabular}\endgroup%
\kern3pt%
\begingroup \smaller\smaller\smaller\begin{tabular}{@{}c@{}}%
6\\1\\-1
\end{tabular}\endgroup%
\kern3pt%
\begingroup \smaller\smaller\smaller\begin{tabular}{@{}c@{}}%
28\\-4\\-2
\end{tabular}\endgroup%
{$\left.\llap{\phantom{%
\begingroup \smaller\smaller\smaller\begin{tabular}{@{}c@{}}%
0\\0\\0
\end{tabular}\endgroup%
}}\!\right]$}%
}%
\ifdim\wd\matricesbox>\halfwidth\myboxwidth=\hsize\else\myboxwidth=\halfwidth\fi
\vbox{%
\ifdim\myboxwidth=\hsize
\setbox\onelinebox=\hbox{%
\vbox{\hbox{%
$\Pi_{4,74}$ spans $L_{22.8}$%
}\hbox{%
$6222$ (shared)%
}%
}%
\hfill\copy\matricesbox
}%
\ifdim\wd\onelinebox>\myboxwidth
\hbox to \myboxwidth{%
$\Pi_{4,74}$ spans $L_{22.8}$%
\hfil
$6222$ (shared)%
}%
\box\matricesbox
\else
\hbox to \myboxwidth{%
\unhbox\onelinebox
}%
\fi
\else
\hbox to \myboxwidth{%
$\Pi_{4,74}$ spans $L_{22.8}$%
\hfil}%
\hbox to \myboxwidth{%
$6222$ (shared)%
\hfil}%
\box\matricesbox
\fi
}%
\hfill\discretionary{}{}{}%
\setbox\matricesbox=\hbox{%
{$\left[\!\llap{\phantom{%
\begingroup \smaller\smaller\smaller\begin{tabular}{@{}c@{}}%
\phantom{0}\\\phantom{0}\\\phantom{0}
\end{tabular}\endgroup%
}}\right.$}%
\begingroup \smaller\smaller\smaller\begin{tabular}{@{}c@{}}%
-1/184\\\phantom{0}\\\phantom{0}
\end{tabular}\endgroup%
\kern3pt%
\begingroup \smaller\smaller\smaller\begin{tabular}{@{}c@{}}%
\phantom{0}\\165/46\\-45/46
\end{tabular}\endgroup%
\kern3pt%
\begingroup \smaller\smaller\smaller\begin{tabular}{@{}c@{}}%
\phantom{0}\\-45/46\\765/46
\end{tabular}\endgroup%
{$\left.\llap{\phantom{%
\begingroup \smaller\smaller\smaller\begin{tabular}{@{}c@{}}%
\phantom{0}\\\phantom{0}\\\phantom{0}
\end{tabular}\endgroup%
}}\!\right]$}%
{$\left[\!\llap{\phantom{%
\begingroup \smaller\smaller\smaller\begin{tabular}{@{}c@{}}%
0\\0\\0
\end{tabular}\endgroup%
}}\right.$}%
\begingroup \smaller\smaller\smaller\begin{tabular}{@{}c@{}}%
90\\-6\\-1
\end{tabular}\endgroup%
\kern3pt%
\begingroup \smaller\smaller\smaller\begin{tabular}{@{}c@{}}%
30\\2\\-1
\end{tabular}\endgroup%
\kern3pt%
\begingroup \smaller\smaller\smaller\begin{tabular}{@{}c@{}}%
36\\3\\1
\end{tabular}\endgroup%
\kern3pt%
\begingroup \smaller\smaller\smaller\begin{tabular}{@{}c@{}}%
20\\-1\\1
\end{tabular}\endgroup%
{$\left.\llap{\phantom{%
\begingroup \smaller\smaller\smaller\begin{tabular}{@{}c@{}}%
0\\0\\0
\end{tabular}\endgroup%
}}\!\right]$}%
}%
\ifdim\wd\matricesbox>\halfwidth\myboxwidth=\hsize\else\myboxwidth=\halfwidth\fi
\vbox{%
\ifdim\myboxwidth=\hsize
\setbox\onelinebox=\hbox{%
\vbox{\hbox{%
$\Pi_{4,75}$ spans $L_{16.20}$%
}\hbox{%
$6222$ (shared)%
}%
}%
\hfill\copy\matricesbox
}%
\ifdim\wd\onelinebox>\myboxwidth
\hbox to \myboxwidth{%
$\Pi_{4,75}$ spans $L_{16.20}$%
\hfil
$6222$ (shared)%
}%
\box\matricesbox
\else
\hbox to \myboxwidth{%
\unhbox\onelinebox
}%
\fi
\else
\hbox to \myboxwidth{%
$\Pi_{4,75}$ spans $L_{16.20}$%
\hfil}%
\hbox to \myboxwidth{%
$6222$ (shared)%
\hfil}%
\box\matricesbox
\fi
}%
\hfill\discretionary{}{}{}%
\setbox\matricesbox=\hbox{%
{$\left[\!\llap{\phantom{%
\begingroup \smaller\smaller\smaller\begin{tabular}{@{}c@{}}%
\phantom{0}\\\phantom{0}\\\phantom{0}
\end{tabular}\endgroup%
}}\right.$}%
\begingroup \smaller\smaller\smaller\begin{tabular}{@{}c@{}}%
-1/58\\\phantom{0}\\\phantom{0}
\end{tabular}\endgroup%
\kern3pt%
\begingroup \smaller\smaller\smaller\begin{tabular}{@{}c@{}}%
\phantom{0}\\40/29\\-15/29
\end{tabular}\endgroup%
\kern3pt%
\begingroup \smaller\smaller\smaller\begin{tabular}{@{}c@{}}%
\phantom{0}\\-15/29\\60/29
\end{tabular}\endgroup%
{$\left.\llap{\phantom{%
\begingroup \smaller\smaller\smaller\begin{tabular}{@{}c@{}}%
\phantom{0}\\\phantom{0}\\\phantom{0}
\end{tabular}\endgroup%
}}\!\right]$}%
{$\left[\!\llap{\phantom{%
\begingroup \smaller\smaller\smaller\begin{tabular}{@{}c@{}}%
0\\0\\0
\end{tabular}\endgroup%
}}\right.$}%
\begingroup \smaller\smaller\smaller\begin{tabular}{@{}c@{}}%
30\\6\\1
\end{tabular}\endgroup%
\kern3pt%
\begingroup \smaller\smaller\smaller\begin{tabular}{@{}c@{}}%
10\\-1\\2
\end{tabular}\endgroup%
\kern3pt%
\begingroup \smaller\smaller\smaller\begin{tabular}{@{}c@{}}%
30\\-6\\-2
\end{tabular}\endgroup%
\kern3pt%
\begingroup \smaller\smaller\smaller\begin{tabular}{@{}c@{}}%
2\\0\\-1
\end{tabular}\endgroup%
{$\left.\llap{\phantom{%
\begingroup \smaller\smaller\smaller\begin{tabular}{@{}c@{}}%
0\\0\\0
\end{tabular}\endgroup%
}}\!\right]$}%
}%
\ifdim\wd\matricesbox>\halfwidth\myboxwidth=\hsize\else\myboxwidth=\halfwidth\fi
\vbox{%
\ifdim\myboxwidth=\hsize
\setbox\onelinebox=\hbox{%
\vbox{\hbox{%
$\Pi_{4,76}$ spans $L_{31.5}$%
}\hbox{%
$6222$ (shared)%
}%
}%
\hfill\copy\matricesbox
}%
\ifdim\wd\onelinebox>\myboxwidth
\hbox to \myboxwidth{%
$\Pi_{4,76}$ spans $L_{31.5}$%
\hfil
$6222$ (shared)%
}%
\box\matricesbox
\else
\hbox to \myboxwidth{%
\unhbox\onelinebox
}%
\fi
\else
\hbox to \myboxwidth{%
$\Pi_{4,76}$ spans $L_{31.5}$%
\hfil}%
\hbox to \myboxwidth{%
$6222$ (shared)%
\hfil}%
\box\matricesbox
\fi
}%
\hfill\discretionary{}{}{}%
\setbox\matricesbox=\hbox{%
{$\left[\!\llap{\phantom{%
\begingroup \smaller\smaller\smaller\begin{tabular}{@{}c@{}}%
\phantom{0}\\\phantom{0}\\\phantom{0}
\end{tabular}\endgroup%
}}\right.$}%
\begingroup \smaller\smaller\smaller\begin{tabular}{@{}c@{}}%
-1/121\\\phantom{0}\\\phantom{0}
\end{tabular}\endgroup%
\kern3pt%
\begingroup \smaller\smaller\smaller\begin{tabular}{@{}c@{}}%
\phantom{0}\\210/121\\-15/121
\end{tabular}\endgroup%
\kern3pt%
\begingroup \smaller\smaller\smaller\begin{tabular}{@{}c@{}}%
\phantom{0}\\-15/121\\390/121
\end{tabular}\endgroup%
{$\left.\llap{\phantom{%
\begingroup \smaller\smaller\smaller\begin{tabular}{@{}c@{}}%
\phantom{0}\\\phantom{0}\\\phantom{0}
\end{tabular}\endgroup%
}}\!\right]$}%
{$\left[\!\llap{\phantom{%
\begingroup \smaller\smaller\smaller\begin{tabular}{@{}c@{}}%
0\\0\\0
\end{tabular}\endgroup%
}}\right.$}%
\begingroup \smaller\smaller\smaller\begin{tabular}{@{}c@{}}%
90\\7\\5
\end{tabular}\endgroup%
\kern3pt%
\begingroup \smaller\smaller\smaller\begin{tabular}{@{}c@{}}%
30\\2\\-3
\end{tabular}\endgroup%
\kern3pt%
\begingroup \smaller\smaller\smaller\begin{tabular}{@{}c@{}}%
9\\-2\\-1
\end{tabular}\endgroup%
\kern3pt%
\begingroup \smaller\smaller\smaller\begin{tabular}{@{}c@{}}%
5\\-1\\1
\end{tabular}\endgroup%
{$\left.\llap{\phantom{%
\begingroup \smaller\smaller\smaller\begin{tabular}{@{}c@{}}%
0\\0\\0
\end{tabular}\endgroup%
}}\!\right]$}%
}%
\ifdim\wd\matricesbox>\halfwidth\myboxwidth=\hsize\else\myboxwidth=\halfwidth\fi
\vbox{%
\ifdim\myboxwidth=\hsize
\setbox\onelinebox=\hbox{%
\vbox{\hbox{%
$\Pi_{4,77}$ spans $L_{16.13}$%
}\hbox{%
$6222$ (shared)%
}%
}%
\hfill\copy\matricesbox
}%
\ifdim\wd\onelinebox>\myboxwidth
\hbox to \myboxwidth{%
$\Pi_{4,77}$ spans $L_{16.13}$%
\hfil
$6222$ (shared)%
}%
\box\matricesbox
\else
\hbox to \myboxwidth{%
\unhbox\onelinebox
}%
\fi
\else
\hbox to \myboxwidth{%
$\Pi_{4,77}$ spans $L_{16.13}$%
\hfil}%
\hbox to \myboxwidth{%
$6222$ (shared)%
\hfil}%
\box\matricesbox
\fi
}%
\hfill\discretionary{}{}{}%
\setbox\matricesbox=\hbox{%
{$\left[\!\llap{\phantom{%
\begingroup \smaller\smaller\smaller\begin{tabular}{@{}c@{}}%
\phantom{0}\\\phantom{0}\\\phantom{0}
\end{tabular}\endgroup%
}}\right.$}%
\begingroup \smaller\smaller\smaller\begin{tabular}{@{}c@{}}%
-1/14\\\phantom{0}\\\phantom{0}
\end{tabular}\endgroup%
\kern3pt%
\begingroup \smaller\smaller\smaller\begin{tabular}{@{}c@{}}%
\phantom{0}\\15/14\\-3/7
\end{tabular}\endgroup%
\kern3pt%
\begingroup \smaller\smaller\smaller\begin{tabular}{@{}c@{}}%
\phantom{0}\\-3/7\\18/7
\end{tabular}\endgroup%
{$\left.\llap{\phantom{%
\begingroup \smaller\smaller\smaller\begin{tabular}{@{}c@{}}%
\phantom{0}\\\phantom{0}\\\phantom{0}
\end{tabular}\endgroup%
}}\!\right]$}%
{$\left[\!\llap{\phantom{%
\begingroup \smaller\smaller\smaller\begin{tabular}{@{}c@{}}%
0\\0\\0
\end{tabular}\endgroup%
}}\right.$}%
\begingroup \smaller\smaller\smaller\begin{tabular}{@{}c@{}}%
1\\1\\0
\end{tabular}\endgroup%
\kern3pt%
\begingroup \smaller\smaller\smaller\begin{tabular}{@{}c@{}}%
4\\-2\\-1
\end{tabular}\endgroup%
\kern3pt%
\begingroup \smaller\smaller\smaller\begin{tabular}{@{}c@{}}%
6\\-2\\1
\end{tabular}\endgroup%
\kern3pt%
\begingroup \smaller\smaller\smaller\begin{tabular}{@{}c@{}}%
12\\2\\3
\end{tabular}\endgroup%
{$\left.\llap{\phantom{%
\begingroup \smaller\smaller\smaller\begin{tabular}{@{}c@{}}%
0\\0\\0
\end{tabular}\endgroup%
}}\!\right]$}%
}%
\ifdim\wd\matricesbox>\halfwidth\myboxwidth=\hsize\else\myboxwidth=\halfwidth\fi
\vbox{%
\ifdim\myboxwidth=\hsize
\setbox\onelinebox=\hbox{%
\vbox{\hbox{%
$\Pi_{4,78}$ spans $L_{4.17}$%
}\hbox{%
$\infty222$ (shared)%
}%
}%
\hfill\copy\matricesbox
}%
\ifdim\wd\onelinebox>\myboxwidth
\hbox to \myboxwidth{%
$\Pi_{4,78}$ spans $L_{4.17}$%
\hfil
$\infty222$ (shared)%
}%
\box\matricesbox
\else
\hbox to \myboxwidth{%
\unhbox\onelinebox
}%
\fi
\else
\hbox to \myboxwidth{%
$\Pi_{4,78}$ spans $L_{4.17}$%
\hfil}%
\hbox to \myboxwidth{%
$\infty222$ (shared)%
\hfil}%
\box\matricesbox
\fi
}%
\hfill\discretionary{}{}{}%
\setbox\matricesbox=\hbox{%
{$\left[\!\llap{\phantom{%
\begingroup \smaller\smaller\smaller\begin{tabular}{@{}c@{}}%
\phantom{0}\\\phantom{0}\\\phantom{0}
\end{tabular}\endgroup%
}}\right.$}%
\begingroup \smaller\smaller\smaller\begin{tabular}{@{}c@{}}%
-1/24\\\phantom{0}\\\phantom{0}
\end{tabular}\endgroup%
\kern3pt%
\begingroup \smaller\smaller\smaller\begin{tabular}{@{}c@{}}%
\phantom{0}\\13/6\\-1
\end{tabular}\endgroup%
\kern3pt%
\begingroup \smaller\smaller\smaller\begin{tabular}{@{}c@{}}%
\phantom{0}\\-1\\6
\end{tabular}\endgroup%
{$\left.\llap{\phantom{%
\begingroup \smaller\smaller\smaller\begin{tabular}{@{}c@{}}%
\phantom{0}\\\phantom{0}\\\phantom{0}
\end{tabular}\endgroup%
}}\!\right]$}%
{$\left[\!\llap{\phantom{%
\begingroup \smaller\smaller\smaller\begin{tabular}{@{}c@{}}%
0\\0\\0
\end{tabular}\endgroup%
}}\right.$}%
\begingroup \smaller\smaller\smaller\begin{tabular}{@{}c@{}}%
2\\1\\0
\end{tabular}\endgroup%
\kern3pt%
\begingroup \smaller\smaller\smaller\begin{tabular}{@{}c@{}}%
8\\-2\\-1
\end{tabular}\endgroup%
\kern3pt%
\begingroup \smaller\smaller\smaller\begin{tabular}{@{}c@{}}%
18\\-3\\1
\end{tabular}\endgroup%
\kern3pt%
\begingroup \smaller\smaller\smaller\begin{tabular}{@{}c@{}}%
72\\6\\7
\end{tabular}\endgroup%
{$\left.\llap{\phantom{%
\begingroup \smaller\smaller\smaller\begin{tabular}{@{}c@{}}%
0\\0\\0
\end{tabular}\endgroup%
}}\!\right]$}%
}%
\ifdim\wd\matricesbox>\halfwidth\myboxwidth=\hsize\else\myboxwidth=\halfwidth\fi
\vbox{%
\ifdim\myboxwidth=\hsize
\setbox\onelinebox=\hbox{%
\vbox{\hbox{%
$\Pi_{4,79}=\hbox{GN}_{9}$ spans $L_{148.13}$%
}\hbox{%
$\infty2\infty2$ (shared)%
}%
}%
\hfill\copy\matricesbox
}%
\ifdim\wd\onelinebox>\myboxwidth
\hbox to \myboxwidth{%
$\Pi_{4,79}=\hbox{GN}_{9}$ spans $L_{148.13}$%
\hfil
$\infty2\infty2$ (shared)%
}%
\box\matricesbox
\else
\hbox to \myboxwidth{%
\unhbox\onelinebox
}%
\fi
\else
\hbox to \myboxwidth{%
$\Pi_{4,79}=\hbox{GN}_{9}$ spans $L_{148.13}$%
\hfil}%
\hbox to \myboxwidth{%
$\infty2\infty2$ (shared)%
\hfil}%
\box\matricesbox
\fi
}%
\hfill\discretionary{}{}{}%
\setbox\matricesbox=\hbox{%
{$\left[\!\llap{\phantom{%
\begingroup \smaller\smaller\smaller\begin{tabular}{@{}c@{}}%
\phantom{0}\\\phantom{0}\\\phantom{0}
\end{tabular}\endgroup%
}}\right.$}%
\begingroup \smaller\smaller\smaller\begin{tabular}{@{}c@{}}%
-1/16\\\phantom{0}\\\phantom{0}
\end{tabular}\endgroup%
\kern3pt%
\begingroup \smaller\smaller\smaller\begin{tabular}{@{}c@{}}%
\phantom{0}\\9/16\\-3/16
\end{tabular}\endgroup%
\kern3pt%
\begingroup \smaller\smaller\smaller\begin{tabular}{@{}c@{}}%
\phantom{0}\\-3/16\\33/16
\end{tabular}\endgroup%
{$\left.\llap{\phantom{%
\begingroup \smaller\smaller\smaller\begin{tabular}{@{}c@{}}%
\phantom{0}\\\phantom{0}\\\phantom{0}
\end{tabular}\endgroup%
}}\!\right]$}%
{$\left[\!\llap{\phantom{%
\begingroup \smaller\smaller\smaller\begin{tabular}{@{}c@{}}%
0\\0\\0
\end{tabular}\endgroup%
}}\right.$}%
\begingroup \smaller\smaller\smaller\begin{tabular}{@{}c@{}}%
2\\1\\1
\end{tabular}\endgroup%
\kern3pt%
\begingroup \smaller\smaller\smaller\begin{tabular}{@{}c@{}}%
8\\2\\-2
\end{tabular}\endgroup%
\kern3pt%
\begingroup \smaller\smaller\smaller\begin{tabular}{@{}c@{}}%
3\\-2\\-1
\end{tabular}\endgroup%
\kern3pt%
\begingroup \smaller\smaller\smaller\begin{tabular}{@{}c@{}}%
6\\-3\\1
\end{tabular}\endgroup%
{$\left.\llap{\phantom{%
\begingroup \smaller\smaller\smaller\begin{tabular}{@{}c@{}}%
0\\0\\0
\end{tabular}\endgroup%
}}\!\right]$}%
}%
\ifdim\wd\matricesbox>\halfwidth\myboxwidth=\hsize\else\myboxwidth=\halfwidth\fi
\vbox{%
\ifdim\myboxwidth=\hsize
\setbox\onelinebox=\hbox{%
\vbox{\hbox{%
$\Pi_{4,80}$ spans $L_{4.15}$%
}\hbox{%
$\infty222$ (shared)%
}%
}%
\hfill\copy\matricesbox
}%
\ifdim\wd\onelinebox>\myboxwidth
\hbox to \myboxwidth{%
$\Pi_{4,80}$ spans $L_{4.15}$%
\hfil
$\infty222$ (shared)%
}%
\box\matricesbox
\else
\hbox to \myboxwidth{%
\unhbox\onelinebox
}%
\fi
\else
\hbox to \myboxwidth{%
$\Pi_{4,80}$ spans $L_{4.15}$%
\hfil}%
\hbox to \myboxwidth{%
$\infty222$ (shared)%
\hfil}%
\box\matricesbox
\fi
}%
\hfill\discretionary{}{}{}%
\setbox\matricesbox=\hbox{%
{$\left[\!\llap{\phantom{%
\begingroup \smaller\smaller\smaller\begin{tabular}{@{}c@{}}%
\phantom{0}\\\phantom{0}\\\phantom{0}
\end{tabular}\endgroup%
}}\right.$}%
\begingroup \smaller\smaller\smaller\begin{tabular}{@{}c@{}}%
-1/13\\\phantom{0}\\\phantom{0}
\end{tabular}\endgroup%
\kern3pt%
\begingroup \smaller\smaller\smaller\begin{tabular}{@{}c@{}}%
\phantom{0}\\10/13\\-4/13
\end{tabular}\endgroup%
\kern3pt%
\begingroup \smaller\smaller\smaller\begin{tabular}{@{}c@{}}%
\phantom{0}\\-4/13\\12/13
\end{tabular}\endgroup%
{$\left.\llap{\phantom{%
\begingroup \smaller\smaller\smaller\begin{tabular}{@{}c@{}}%
\phantom{0}\\\phantom{0}\\\phantom{0}
\end{tabular}\endgroup%
}}\!\right]$}%
{$\left[\!\llap{\phantom{%
\begingroup \smaller\smaller\smaller\begin{tabular}{@{}c@{}}%
0\\0\\0
\end{tabular}\endgroup%
}}\right.$}%
\begingroup \smaller\smaller\smaller\begin{tabular}{@{}c@{}}%
2\\-1\\1
\end{tabular}\endgroup%
\kern3pt%
\begingroup \smaller\smaller\smaller\begin{tabular}{@{}c@{}}%
8\\-2\\-4
\end{tabular}\endgroup%
\kern3pt%
\begingroup \smaller\smaller\smaller\begin{tabular}{@{}c@{}}%
4\\2\\-1
\end{tabular}\endgroup%
\kern3pt%
\begingroup \smaller\smaller\smaller\begin{tabular}{@{}c@{}}%
1\\1\\1
\end{tabular}\endgroup%
{$\left.\llap{\phantom{%
\begingroup \smaller\smaller\smaller\begin{tabular}{@{}c@{}}%
0\\0\\0
\end{tabular}\endgroup%
}}\!\right]$}%
}%
\ifdim\wd\matricesbox>\halfwidth\myboxwidth=\hsize\else\myboxwidth=\halfwidth\fi
\vbox{%
\ifdim\myboxwidth=\hsize
\setbox\onelinebox=\hbox{%
\vbox{\hbox{%
$\Pi_{4,81}$ spans $L_{141.7}$%
}\hbox{%
$\infty222$ (shared)%
}%
}%
\hfill\copy\matricesbox
}%
\ifdim\wd\onelinebox>\myboxwidth
\hbox to \myboxwidth{%
$\Pi_{4,81}$ spans $L_{141.7}$%
\hfil
$\infty222$ (shared)%
}%
\box\matricesbox
\else
\hbox to \myboxwidth{%
\unhbox\onelinebox
}%
\fi
\else
\hbox to \myboxwidth{%
$\Pi_{4,81}$ spans $L_{141.7}$%
\hfil}%
\hbox to \myboxwidth{%
$\infty222$ (shared)%
\hfil}%
\box\matricesbox
\fi
}%
\hfill\discretionary{}{}{}%
\setbox\matricesbox=\hbox{%
{$\left[\!\llap{\phantom{%
\begingroup \smaller\smaller\smaller\begin{tabular}{@{}c@{}}%
\phantom{0}\\\phantom{0}\\\phantom{0}
\end{tabular}\endgroup%
}}\right.$}%
\begingroup \smaller\smaller\smaller\begin{tabular}{@{}c@{}}%
-1/18\\\phantom{0}\\\phantom{0}
\end{tabular}\endgroup%
\kern3pt%
\begingroup \smaller\smaller\smaller\begin{tabular}{@{}c@{}}%
\phantom{0}\\7/18\\-1/18
\end{tabular}\endgroup%
\kern3pt%
\begingroup \smaller\smaller\smaller\begin{tabular}{@{}c@{}}%
\phantom{0}\\-1/18\\31/18
\end{tabular}\endgroup%
{$\left.\llap{\phantom{%
\begingroup \smaller\smaller\smaller\begin{tabular}{@{}c@{}}%
\phantom{0}\\\phantom{0}\\\phantom{0}
\end{tabular}\endgroup%
}}\!\right]$}%
{$\left[\!\llap{\phantom{%
\begingroup \smaller\smaller\smaller\begin{tabular}{@{}c@{}}%
0\\0\\0
\end{tabular}\endgroup%
}}\right.$}%
\begingroup \smaller\smaller\smaller\begin{tabular}{@{}c@{}}%
3\\-2\\1
\end{tabular}\endgroup%
\kern3pt%
\begingroup \smaller\smaller\smaller\begin{tabular}{@{}c@{}}%
12\\7\\1
\end{tabular}\endgroup%
\kern3pt%
\begingroup \smaller\smaller\smaller\begin{tabular}{@{}c@{}}%
2\\1\\-1
\end{tabular}\endgroup%
\kern3pt%
\begingroup \smaller\smaller\smaller\begin{tabular}{@{}c@{}}%
4\\-3\\-1
\end{tabular}\endgroup%
{$\left.\llap{\phantom{%
\begingroup \smaller\smaller\smaller\begin{tabular}{@{}c@{}}%
0\\0\\0
\end{tabular}\endgroup%
}}\!\right]$}%
}%
\ifdim\wd\matricesbox>\halfwidth\myboxwidth=\hsize\else\myboxwidth=\halfwidth\fi
\vbox{%
\ifdim\myboxwidth=\hsize
\setbox\onelinebox=\hbox{%
\vbox{\hbox{%
$\Pi_{4,82}$ spans $L_{4.6}$%
}\hbox{%
$\infty222$ (shared)%
}%
}%
\hfill\copy\matricesbox
}%
\ifdim\wd\onelinebox>\myboxwidth
\hbox to \myboxwidth{%
$\Pi_{4,82}$ spans $L_{4.6}$%
\hfil
$\infty222$ (shared)%
}%
\box\matricesbox
\else
\hbox to \myboxwidth{%
\unhbox\onelinebox
}%
\fi
\else
\hbox to \myboxwidth{%
$\Pi_{4,82}$ spans $L_{4.6}$%
\hfil}%
\hbox to \myboxwidth{%
$\infty222$ (shared)%
\hfil}%
\box\matricesbox
\fi
}%
\hfill\discretionary{}{}{}%
\setbox\matricesbox=\hbox{%
{$\left[\!\llap{\phantom{%
\begingroup \smaller\smaller\smaller\begin{tabular}{@{}c@{}}%
\phantom{0}\\\phantom{0}\\\phantom{0}
\end{tabular}\endgroup%
}}\right.$}%
\begingroup \smaller\smaller\smaller\begin{tabular}{@{}c@{}}%
-1/56\\\phantom{0}\\\phantom{0}
\end{tabular}\endgroup%
\kern3pt%
\begingroup \smaller\smaller\smaller\begin{tabular}{@{}c@{}}%
\phantom{0}\\3/2\\-1/2
\end{tabular}\endgroup%
\kern3pt%
\begingroup \smaller\smaller\smaller\begin{tabular}{@{}c@{}}%
\phantom{0}\\-1/2\\29/14
\end{tabular}\endgroup%
{$\left.\llap{\phantom{%
\begingroup \smaller\smaller\smaller\begin{tabular}{@{}c@{}}%
\phantom{0}\\\phantom{0}\\\phantom{0}
\end{tabular}\endgroup%
}}\!\right]$}%
{$\left[\!\llap{\phantom{%
\begingroup \smaller\smaller\smaller\begin{tabular}{@{}c@{}}%
0\\0\\0
\end{tabular}\endgroup%
}}\right.$}%
\begingroup \smaller\smaller\smaller\begin{tabular}{@{}c@{}}%
10\\-1\\2
\end{tabular}\endgroup%
\kern3pt%
\begingroup \smaller\smaller\smaller\begin{tabular}{@{}c@{}}%
40\\7\\1
\end{tabular}\endgroup%
\kern3pt%
\begingroup \smaller\smaller\smaller\begin{tabular}{@{}c@{}}%
2\\0\\-1
\end{tabular}\endgroup%
\kern3pt%
\begingroup \smaller\smaller\smaller\begin{tabular}{@{}c@{}}%
32\\-6\\-2
\end{tabular}\endgroup%
{$\left.\llap{\phantom{%
\begingroup \smaller\smaller\smaller\begin{tabular}{@{}c@{}}%
0\\0\\0
\end{tabular}\endgroup%
}}\!\right]$}%
}%
\ifdim\wd\matricesbox>\halfwidth\myboxwidth=\hsize\else\myboxwidth=\halfwidth\fi
\vbox{%
\ifdim\myboxwidth=\hsize
\setbox\onelinebox=\hbox{%
\vbox{\hbox{%
$\Pi_{4,83}$ spans $L_{10.3}$%
}\hbox{%
$\infty222$ (shared)%
}%
}%
\hfill\copy\matricesbox
}%
\ifdim\wd\onelinebox>\myboxwidth
\hbox to \myboxwidth{%
$\Pi_{4,83}$ spans $L_{10.3}$%
\hfil
$\infty222$ (shared)%
}%
\box\matricesbox
\else
\hbox to \myboxwidth{%
\unhbox\onelinebox
}%
\fi
\else
\hbox to \myboxwidth{%
$\Pi_{4,83}$ spans $L_{10.3}$%
\hfil}%
\hbox to \myboxwidth{%
$\infty222$ (shared)%
\hfil}%
\box\matricesbox
\fi
}%
\hfill\discretionary{}{}{}%
\setbox\matricesbox=\hbox{%
{$\left[\!\llap{\phantom{%
\begingroup \smaller\smaller\smaller\begin{tabular}{@{}c@{}}%
\phantom{0}\\\phantom{0}\\\phantom{0}
\end{tabular}\endgroup%
}}\right.$}%
\begingroup \smaller\smaller\smaller\begin{tabular}{@{}c@{}}%
-1/24\\\phantom{0}\\\phantom{0}
\end{tabular}\endgroup%
\kern3pt%
\begingroup \smaller\smaller\smaller\begin{tabular}{@{}c@{}}%
\phantom{0}\\3/8\\\phantom{0}
\end{tabular}\endgroup%
\kern3pt%
\begingroup \smaller\smaller\smaller\begin{tabular}{@{}c@{}}%
\phantom{0}\\\phantom{0}\\2/3
\end{tabular}\endgroup%
{$\left.\llap{\phantom{%
\begingroup \smaller\smaller\smaller\begin{tabular}{@{}c@{}}%
\phantom{0}\\\phantom{0}\\\phantom{0}
\end{tabular}\endgroup%
}}\!\right]$}%
{$\left[\!\llap{\phantom{%
\begingroup \smaller\smaller\smaller\begin{tabular}{@{}c@{}}%
0\\0\\0
\end{tabular}\endgroup%
}}\right.$}%
\begingroup \smaller\smaller\smaller\begin{tabular}{@{}c@{}}%
6\\-2\\-3
\end{tabular}\endgroup%
\kern3pt%
\begingroup \smaller\smaller\smaller\begin{tabular}{@{}c@{}}%
24\\-8\\6
\end{tabular}\endgroup%
\kern3pt%
\begingroup \smaller\smaller\smaller\begin{tabular}{@{}c@{}}%
1\\1\\1
\end{tabular}\endgroup%
\kern3pt%
\begingroup \smaller\smaller\smaller\begin{tabular}{@{}c@{}}%
2\\2\\-1
\end{tabular}\endgroup%
{$\left.\llap{\phantom{%
\begingroup \smaller\smaller\smaller\begin{tabular}{@{}c@{}}%
0\\0\\0
\end{tabular}\endgroup%
}}\!\right]$}%
}%
\ifdim\wd\matricesbox>\halfwidth\myboxwidth=\hsize\else\myboxwidth=\halfwidth\fi
\vbox{%
\ifdim\myboxwidth=\hsize
\setbox\onelinebox=\hbox{%
\vbox{\hbox{%
$\Pi_{4,84}$ spans $L_{4.5}$%
}\hbox{%
$\infty222$ (shared)%
}%
}%
\hfill\copy\matricesbox
}%
\ifdim\wd\onelinebox>\myboxwidth
\hbox to \myboxwidth{%
$\Pi_{4,84}$ spans $L_{4.5}$%
\hfil
$\infty222$ (shared)%
}%
\box\matricesbox
\else
\hbox to \myboxwidth{%
\unhbox\onelinebox
}%
\fi
\else
\hbox to \myboxwidth{%
$\Pi_{4,84}$ spans $L_{4.5}$%
\hfil}%
\hbox to \myboxwidth{%
$\infty222$ (shared)%
\hfil}%
\box\matricesbox
\fi
}%
\hfill\discretionary{}{}{}%
\setbox\matricesbox=\hbox{%
{$\left[\!\llap{\phantom{%
\begingroup \smaller\smaller\smaller\begin{tabular}{@{}c@{}}%
\phantom{0}\\\phantom{0}\\\phantom{0}
\end{tabular}\endgroup%
}}\right.$}%
\begingroup \smaller\smaller\smaller\begin{tabular}{@{}c@{}}%
-1/13\\\phantom{0}\\\phantom{0}
\end{tabular}\endgroup%
\kern3pt%
\begingroup \smaller\smaller\smaller\begin{tabular}{@{}c@{}}%
\phantom{0}\\12/13\\-3/13
\end{tabular}\endgroup%
\kern3pt%
\begingroup \smaller\smaller\smaller\begin{tabular}{@{}c@{}}%
\phantom{0}\\-3/13\\30/13
\end{tabular}\endgroup%
{$\left.\llap{\phantom{%
\begingroup \smaller\smaller\smaller\begin{tabular}{@{}c@{}}%
\phantom{0}\\\phantom{0}\\\phantom{0}
\end{tabular}\endgroup%
}}\!\right]$}%
{$\left[\!\llap{\phantom{%
\begingroup \smaller\smaller\smaller\begin{tabular}{@{}c@{}}%
0\\0\\0
\end{tabular}\endgroup%
}}\right.$}%
\begingroup \smaller\smaller\smaller\begin{tabular}{@{}c@{}}%
3\\1\\-1
\end{tabular}\endgroup%
\kern3pt%
\begingroup \smaller\smaller\smaller\begin{tabular}{@{}c@{}}%
3\\-2\\0
\end{tabular}\endgroup%
\kern3pt%
\begingroup \smaller\smaller\smaller\begin{tabular}{@{}c@{}}%
2\\0\\1
\end{tabular}\endgroup%
\kern3pt%
\begingroup \smaller\smaller\smaller\begin{tabular}{@{}c@{}}%
9\\4\\1
\end{tabular}\endgroup%
{$\left.\llap{\phantom{%
\begingroup \smaller\smaller\smaller\begin{tabular}{@{}c@{}}%
0\\0\\0
\end{tabular}\endgroup%
}}\!\right]$}%
}%
\ifdim\wd\matricesbox>\halfwidth\myboxwidth=\hsize\else\myboxwidth=\halfwidth\fi
\vbox{%
\ifdim\myboxwidth=\hsize
\setbox\onelinebox=\hbox{%
\vbox{\hbox{%
$\Pi_{4,85}$ spans $L_{4.16}$%
}\hbox{%
$\infty222$ (shared)%
}%
}%
\hfill\copy\matricesbox
}%
\ifdim\wd\onelinebox>\myboxwidth
\hbox to \myboxwidth{%
$\Pi_{4,85}$ spans $L_{4.16}$%
\hfil
$\infty222$ (shared)%
}%
\box\matricesbox
\else
\hbox to \myboxwidth{%
\unhbox\onelinebox
}%
\fi
\else
\hbox to \myboxwidth{%
$\Pi_{4,85}$ spans $L_{4.16}$%
\hfil}%
\hbox to \myboxwidth{%
$\infty222$ (shared)%
\hfil}%
\box\matricesbox
\fi
}%
\hfill\discretionary{}{}{}%
\setbox\matricesbox=\hbox{%
{$\left[\!\llap{\phantom{%
\begingroup \smaller\smaller\smaller\begin{tabular}{@{}c@{}}%
\phantom{0}\\\phantom{0}\\\phantom{0}
\end{tabular}\endgroup%
}}\right.$}%
\begingroup \smaller\smaller\smaller\begin{tabular}{@{}c@{}}%
-1/38\\\phantom{0}\\\phantom{0}
\end{tabular}\endgroup%
\kern3pt%
\begingroup \smaller\smaller\smaller\begin{tabular}{@{}c@{}}%
\phantom{0}\\40/19\\-5/19
\end{tabular}\endgroup%
\kern3pt%
\begingroup \smaller\smaller\smaller\begin{tabular}{@{}c@{}}%
\phantom{0}\\-5/19\\60/19
\end{tabular}\endgroup%
{$\left.\llap{\phantom{%
\begingroup \smaller\smaller\smaller\begin{tabular}{@{}c@{}}%
\phantom{0}\\\phantom{0}\\\phantom{0}
\end{tabular}\endgroup%
}}\!\right]$}%
{$\left[\!\llap{\phantom{%
\begingroup \smaller\smaller\smaller\begin{tabular}{@{}c@{}}%
0\\0\\0
\end{tabular}\endgroup%
}}\right.$}%
\begingroup \smaller\smaller\smaller\begin{tabular}{@{}c@{}}%
10\\2\\-1
\end{tabular}\endgroup%
\kern3pt%
\begingroup \smaller\smaller\smaller\begin{tabular}{@{}c@{}}%
10\\0\\2
\end{tabular}\endgroup%
\kern3pt%
\begingroup \smaller\smaller\smaller\begin{tabular}{@{}c@{}}%
2\\-1\\0
\end{tabular}\endgroup%
\kern3pt%
\begingroup \smaller\smaller\smaller\begin{tabular}{@{}c@{}}%
50\\-2\\-6
\end{tabular}\endgroup%
{$\left.\llap{\phantom{%
\begingroup \smaller\smaller\smaller\begin{tabular}{@{}c@{}}%
0\\0\\0
\end{tabular}\endgroup%
}}\!\right]$}%
}%
\ifdim\wd\matricesbox>\halfwidth\myboxwidth=\hsize\else\myboxwidth=\halfwidth\fi
\vbox{%
\ifdim\myboxwidth=\hsize
\setbox\onelinebox=\hbox{%
\vbox{\hbox{%
$\Pi_{4,86}$ spans $L_{10.9}$%
}\hbox{%
$\infty222$ (shared)%
}%
}%
\hfill\copy\matricesbox
}%
\ifdim\wd\onelinebox>\myboxwidth
\hbox to \myboxwidth{%
$\Pi_{4,86}$ spans $L_{10.9}$%
\hfil
$\infty222$ (shared)%
}%
\box\matricesbox
\else
\hbox to \myboxwidth{%
\unhbox\onelinebox
}%
\fi
\else
\hbox to \myboxwidth{%
$\Pi_{4,86}$ spans $L_{10.9}$%
\hfil}%
\hbox to \myboxwidth{%
$\infty222$ (shared)%
\hfil}%
\box\matricesbox
\fi
}%
\hfill\discretionary{}{}{}%
\setbox\matricesbox=\hbox{%
{$\left[\!\llap{\phantom{%
\begingroup \smaller\smaller\smaller\begin{tabular}{@{}c@{}}%
\phantom{0}\\\phantom{0}\\\phantom{0}
\end{tabular}\endgroup%
}}\right.$}%
\begingroup \smaller\smaller\smaller\begin{tabular}{@{}c@{}}%
-1/18\\\phantom{0}\\\phantom{0}
\end{tabular}\endgroup%
\kern3pt%
\begingroup \smaller\smaller\smaller\begin{tabular}{@{}c@{}}%
\phantom{0}\\7/18\\-1/9
\end{tabular}\endgroup%
\kern3pt%
\begingroup \smaller\smaller\smaller\begin{tabular}{@{}c@{}}%
\phantom{0}\\-1/9\\8/9
\end{tabular}\endgroup%
{$\left.\llap{\phantom{%
\begingroup \smaller\smaller\smaller\begin{tabular}{@{}c@{}}%
\phantom{0}\\\phantom{0}\\\phantom{0}
\end{tabular}\endgroup%
}}\!\right]$}%
{$\left[\!\llap{\phantom{%
\begingroup \smaller\smaller\smaller\begin{tabular}{@{}c@{}}%
0\\0\\0
\end{tabular}\endgroup%
}}\right.$}%
\begingroup \smaller\smaller\smaller\begin{tabular}{@{}c@{}}%
6\\4\\2
\end{tabular}\endgroup%
\kern3pt%
\begingroup \smaller\smaller\smaller\begin{tabular}{@{}c@{}}%
6\\-4\\1
\end{tabular}\endgroup%
\kern3pt%
\begingroup \smaller\smaller\smaller\begin{tabular}{@{}c@{}}%
1\\-1\\-1
\end{tabular}\endgroup%
\kern3pt%
\begingroup \smaller\smaller\smaller\begin{tabular}{@{}c@{}}%
8\\4\\-2
\end{tabular}\endgroup%
{$\left.\llap{\phantom{%
\begingroup \smaller\smaller\smaller\begin{tabular}{@{}c@{}}%
0\\0\\0
\end{tabular}\endgroup%
}}\!\right]$}%
}%
\ifdim\wd\matricesbox>\halfwidth\myboxwidth=\hsize\else\myboxwidth=\halfwidth\fi
\vbox{%
\ifdim\myboxwidth=\hsize
\setbox\onelinebox=\hbox{%
\vbox{\hbox{%
$\Pi_{4,87}$ spans $L_{4.5}$%
}\hbox{%
$\infty222$ (shared)%
}%
}%
\hfill\copy\matricesbox
}%
\ifdim\wd\onelinebox>\myboxwidth
\hbox to \myboxwidth{%
$\Pi_{4,87}$ spans $L_{4.5}$%
\hfil
$\infty222$ (shared)%
}%
\box\matricesbox
\else
\hbox to \myboxwidth{%
\unhbox\onelinebox
}%
\fi
\else
\hbox to \myboxwidth{%
$\Pi_{4,87}$ spans $L_{4.5}$%
\hfil}%
\hbox to \myboxwidth{%
$\infty222$ (shared)%
\hfil}%
\box\matricesbox
\fi
}%
\hfill\discretionary{}{}{}%

\vskip2pt\hrule\vskip2pt

\leavevmode\setbox\matricesbox=\hbox{%
{$\left[\!\llap{\phantom{%
\begingroup \smaller\smaller\smaller\begin{tabular}{@{}c@{}}%
\phantom{0}\\\phantom{0}\\\phantom{0}
\end{tabular}\endgroup%
}}\right.$}%
\begingroup \smaller\smaller\smaller\begin{tabular}{@{}c@{}}%
-1/8\\\phantom{0}\\\phantom{0}
\end{tabular}\endgroup%
\kern3pt%
\begingroup \smaller\smaller\smaller\begin{tabular}{@{}c@{}}%
\phantom{0}\\9/8\\\phantom{0}
\end{tabular}\endgroup%
\kern3pt%
\begingroup \smaller\smaller\smaller\begin{tabular}{@{}c@{}}%
\phantom{0}\\\phantom{0}\\6
\end{tabular}\endgroup%
{$\left.\llap{\phantom{%
\begingroup \smaller\smaller\smaller\begin{tabular}{@{}c@{}}%
\phantom{0}\\\phantom{0}\\\phantom{0}
\end{tabular}\endgroup%
}}\!\right]$}%
{$\left[\!\llap{\phantom{%
\begingroup \smaller\smaller\smaller\begin{tabular}{@{}c@{}}%
0\\0\\0
\end{tabular}\endgroup%
}}\right.$}%
\begingroup \smaller\smaller\smaller\begin{tabular}{@{}c@{}}%
1\\1\\0
\end{tabular}\endgroup%
\kern3pt%
\begingroup \smaller\smaller\smaller\begin{tabular}{@{}c@{}}%
18\\2\\3
\end{tabular}\endgroup%
\kern3pt%
\begingroup \smaller\smaller\smaller\begin{tabular}{@{}c@{}}%
6\\-2\\1
\end{tabular}\endgroup%
{$\left.\llap{\phantom{%
\begingroup \smaller\smaller\smaller\begin{tabular}{@{}c@{}}%
0\\0\\0
\end{tabular}\endgroup%
}}\!\right]$}%
}%
\ifdim\wd\matricesbox>\halfwidth\myboxwidth=\hsize\else\myboxwidth=\halfwidth\fi
\vbox{%
\ifdim\myboxwidth=\hsize
\setbox\onelinebox=\hbox{%
\vbox{\hbox{%
$\Pi_{5,1}$ spans $L_{4.22}$%
}\hbox{%
$|22\slashinfty22\rtimes D_{2}$ (shared)%
}%
}%
\hfill\copy\matricesbox
}%
\ifdim\wd\onelinebox>\myboxwidth
\hbox to \myboxwidth{%
$\Pi_{5,1}$ spans $L_{4.22}$%
\hfil
$|22\slashinfty22\rtimes D_{2}$ (shared)%
}%
\box\matricesbox
\else
\hbox to \myboxwidth{%
\unhbox\onelinebox
}%
\fi
\else
\hbox to \myboxwidth{%
$\Pi_{5,1}$ spans $L_{4.22}$%
\hfil}%
\hbox to \myboxwidth{%
$|22\slashinfty22\rtimes D_{2}$ (shared)%
\hfil}%
\box\matricesbox
\fi
}%
\hfill\discretionary{}{}{}%
\setbox\matricesbox=\hbox{%
{$\left[\!\llap{\phantom{%
\begingroup \smaller\smaller\smaller\begin{tabular}{@{}c@{}}%
\phantom{0}\\\phantom{0}\\\phantom{0}
\end{tabular}\endgroup%
}}\right.$}%
\begingroup \smaller\smaller\smaller\begin{tabular}{@{}c@{}}%
-1/19\\\phantom{0}\\\phantom{0}
\end{tabular}\endgroup%
\kern3pt%
\begingroup \smaller\smaller\smaller\begin{tabular}{@{}c@{}}%
\phantom{0}\\21/38\\\phantom{0}
\end{tabular}\endgroup%
\kern3pt%
\begingroup \smaller\smaller\smaller\begin{tabular}{@{}c@{}}%
\phantom{0}\\\phantom{0}\\21/2
\end{tabular}\endgroup%
{$\left.\llap{\phantom{%
\begingroup \smaller\smaller\smaller\begin{tabular}{@{}c@{}}%
\phantom{0}\\\phantom{0}\\\phantom{0}
\end{tabular}\endgroup%
}}\!\right]$}%
{$\left[\!\llap{\phantom{%
\begingroup \smaller\smaller\smaller\begin{tabular}{@{}c@{}}%
0\\0\\0
\end{tabular}\endgroup%
}}\right.$}%
\begingroup \smaller\smaller\smaller\begin{tabular}{@{}c@{}}%
2\\-2\\0
\end{tabular}\endgroup%
\kern3pt%
\begingroup \smaller\smaller\smaller\begin{tabular}{@{}c@{}}%
21\\-2\\2
\end{tabular}\endgroup%
\kern3pt%
\begingroup \smaller\smaller\smaller\begin{tabular}{@{}c@{}}%
14\\5\\1
\end{tabular}\endgroup%
{$\left.\llap{\phantom{%
\begingroup \smaller\smaller\smaller\begin{tabular}{@{}c@{}}%
0\\0\\0
\end{tabular}\endgroup%
}}\!\right]$}%
}%
\ifdim\wd\matricesbox>\halfwidth\myboxwidth=\hsize\else\myboxwidth=\halfwidth\fi
\vbox{%
\ifdim\myboxwidth=\hsize
\setbox\onelinebox=\hbox{%
\vbox{\hbox{%
$\Pi_{5,2}$ spans $L_{22.7}$%
}\hbox{%
$|22\slashthree22\rtimes D_{2}$%
}%
}%
\hfill\copy\matricesbox
}%
\ifdim\wd\onelinebox>\myboxwidth
\hbox to \myboxwidth{%
$\Pi_{5,2}$ spans $L_{22.7}$%
\hfil
$|22\slashthree22\rtimes D_{2}$%
}%
\box\matricesbox
\else
\hbox to \myboxwidth{%
\unhbox\onelinebox
}%
\fi
\else
\hbox to \myboxwidth{%
$\Pi_{5,2}$ spans $L_{22.7}$%
\hfil}%
\hbox to \myboxwidth{%
$|22\slashthree22\rtimes D_{2}$%
\hfil}%
\box\matricesbox
\fi
}%
\hfill\discretionary{}{}{}%
\setbox\matricesbox=\hbox{%
{$\left[\!\llap{\phantom{%
\begingroup \smaller\smaller\smaller\begin{tabular}{@{}c@{}}%
\phantom{0}\\\phantom{0}\\\phantom{0}
\end{tabular}\endgroup%
}}\right.$}%
\begingroup \smaller\smaller\smaller\begin{tabular}{@{}c@{}}%
-1/9\\\phantom{0}\\\phantom{0}
\end{tabular}\endgroup%
\kern3pt%
\begingroup \smaller\smaller\smaller\begin{tabular}{@{}c@{}}%
\phantom{0}\\5/18\\\phantom{0}
\end{tabular}\endgroup%
\kern3pt%
\begingroup \smaller\smaller\smaller\begin{tabular}{@{}c@{}}%
\phantom{0}\\\phantom{0}\\15/2
\end{tabular}\endgroup%
{$\left.\llap{\phantom{%
\begingroup \smaller\smaller\smaller\begin{tabular}{@{}c@{}}%
\phantom{0}\\\phantom{0}\\\phantom{0}
\end{tabular}\endgroup%
}}\!\right]$}%
{$\left[\!\llap{\phantom{%
\begingroup \smaller\smaller\smaller\begin{tabular}{@{}c@{}}%
0\\0\\0
\end{tabular}\endgroup%
}}\right.$}%
\begingroup \smaller\smaller\smaller\begin{tabular}{@{}c@{}}%
1\\2\\0
\end{tabular}\endgroup%
\kern3pt%
\begingroup \smaller\smaller\smaller\begin{tabular}{@{}c@{}}%
5\\1\\1
\end{tabular}\endgroup%
\kern3pt%
\begingroup \smaller\smaller\smaller\begin{tabular}{@{}c@{}}%
10\\-7\\1
\end{tabular}\endgroup%
{$\left.\llap{\phantom{%
\begingroup \smaller\smaller\smaller\begin{tabular}{@{}c@{}}%
0\\0\\0
\end{tabular}\endgroup%
}}\!\right]$}%
}%
\ifdim\wd\matricesbox>\halfwidth\myboxwidth=\hsize\else\myboxwidth=\halfwidth\fi
\vbox{%
\ifdim\myboxwidth=\hsize
\setbox\onelinebox=\hbox{%
\vbox{\hbox{%
$\Pi_{5,3}$ spans $L_{16.6}$%
}\hbox{%
$|22\slashthree22\rtimes D_{2}$%
}%
}%
\hfill\copy\matricesbox
}%
\ifdim\wd\onelinebox>\myboxwidth
\hbox to \myboxwidth{%
$\Pi_{5,3}$ spans $L_{16.6}$%
\hfil
$|22\slashthree22\rtimes D_{2}$%
}%
\box\matricesbox
\else
\hbox to \myboxwidth{%
\unhbox\onelinebox
}%
\fi
\else
\hbox to \myboxwidth{%
$\Pi_{5,3}$ spans $L_{16.6}$%
\hfil}%
\hbox to \myboxwidth{%
$|22\slashthree22\rtimes D_{2}$%
\hfil}%
\box\matricesbox
\fi
}%
\hfill\discretionary{}{}{}%
\setbox\matricesbox=\hbox{%
{$\left[\!\llap{\phantom{%
\begingroup \smaller\smaller\smaller\begin{tabular}{@{}c@{}}%
\phantom{0}\\\phantom{0}\\\phantom{0}
\end{tabular}\endgroup%
}}\right.$}%
\begingroup \smaller\smaller\smaller\begin{tabular}{@{}c@{}}%
-1/16\\\phantom{0}\\\phantom{0}
\end{tabular}\endgroup%
\kern3pt%
\begingroup \smaller\smaller\smaller\begin{tabular}{@{}c@{}}%
\phantom{0}\\9/16\\\phantom{0}
\end{tabular}\endgroup%
\kern3pt%
\begingroup \smaller\smaller\smaller\begin{tabular}{@{}c@{}}%
\phantom{0}\\\phantom{0}\\3/2
\end{tabular}\endgroup%
{$\left.\llap{\phantom{%
\begingroup \smaller\smaller\smaller\begin{tabular}{@{}c@{}}%
\phantom{0}\\\phantom{0}\\\phantom{0}
\end{tabular}\endgroup%
}}\!\right]$}%
{$\left[\!\llap{\phantom{%
\begingroup \smaller\smaller\smaller\begin{tabular}{@{}c@{}}%
0\\0\\0
\end{tabular}\endgroup%
}}\right.$}%
\begingroup \smaller\smaller\smaller\begin{tabular}{@{}c@{}}%
2\\-2\\0
\end{tabular}\endgroup%
\kern3pt%
\begingroup \smaller\smaller\smaller\begin{tabular}{@{}c@{}}%
9\\-1\\-3
\end{tabular}\endgroup%
\kern3pt%
\begingroup \smaller\smaller\smaller\begin{tabular}{@{}c@{}}%
24\\8\\-4
\end{tabular}\endgroup%
{$\left.\llap{\phantom{%
\begingroup \smaller\smaller\smaller\begin{tabular}{@{}c@{}}%
0\\0\\0
\end{tabular}\endgroup%
}}\!\right]$}%
}%
\ifdim\wd\matricesbox>\halfwidth\myboxwidth=\hsize\else\myboxwidth=\halfwidth\fi
\vbox{%
\ifdim\myboxwidth=\hsize
\setbox\onelinebox=\hbox{%
\vbox{\hbox{%
$\Pi_{5,4}$ spans $L_{4.11}$%
}\hbox{%
$|22\slashinfty22\rtimes D_{2}$ (shared)%
}%
}%
\hfill\copy\matricesbox
}%
\ifdim\wd\onelinebox>\myboxwidth
\hbox to \myboxwidth{%
$\Pi_{5,4}$ spans $L_{4.11}$%
\hfil
$|22\slashinfty22\rtimes D_{2}$ (shared)%
}%
\box\matricesbox
\else
\hbox to \myboxwidth{%
\unhbox\onelinebox
}%
\fi
\else
\hbox to \myboxwidth{%
$\Pi_{5,4}$ spans $L_{4.11}$%
\hfil}%
\hbox to \myboxwidth{%
$|22\slashinfty22\rtimes D_{2}$ (shared)%
\hfil}%
\box\matricesbox
\fi
}%
\hfill\discretionary{}{}{}%
\setbox\matricesbox=\hbox{%
{$\left[\!\llap{\phantom{%
\begingroup \smaller\smaller\smaller\begin{tabular}{@{}c@{}}%
\phantom{0}\\\phantom{0}\\\phantom{0}
\end{tabular}\endgroup%
}}\right.$}%
\begingroup \smaller\smaller\smaller\begin{tabular}{@{}c@{}}%
-1/6\\\phantom{0}\\\phantom{0}
\end{tabular}\endgroup%
\kern3pt%
\begingroup \smaller\smaller\smaller\begin{tabular}{@{}c@{}}%
\phantom{0}\\1/6\\\phantom{0}
\end{tabular}\endgroup%
\kern3pt%
\begingroup \smaller\smaller\smaller\begin{tabular}{@{}c@{}}%
\phantom{0}\\\phantom{0}\\3
\end{tabular}\endgroup%
{$\left.\llap{\phantom{%
\begingroup \smaller\smaller\smaller\begin{tabular}{@{}c@{}}%
\phantom{0}\\\phantom{0}\\\phantom{0}
\end{tabular}\endgroup%
}}\!\right]$}%
{$\left[\!\llap{\phantom{%
\begingroup \smaller\smaller\smaller\begin{tabular}{@{}c@{}}%
0\\0\\0
\end{tabular}\endgroup%
}}\right.$}%
\begingroup \smaller\smaller\smaller\begin{tabular}{@{}c@{}}%
2\\4\\0
\end{tabular}\endgroup%
\kern3pt%
\begingroup \smaller\smaller\smaller\begin{tabular}{@{}c@{}}%
16\\8\\4
\end{tabular}\endgroup%
\kern3pt%
\begingroup \smaller\smaller\smaller\begin{tabular}{@{}c@{}}%
3\\-3\\1
\end{tabular}\endgroup%
{$\left.\llap{\phantom{%
\begingroup \smaller\smaller\smaller\begin{tabular}{@{}c@{}}%
0\\0\\0
\end{tabular}\endgroup%
}}\!\right]$}%
}%
\ifdim\wd\matricesbox>\halfwidth\myboxwidth=\hsize\else\myboxwidth=\halfwidth\fi
\vbox{%
\ifdim\myboxwidth=\hsize
\setbox\onelinebox=\hbox{%
\vbox{\hbox{%
$\Pi_{5,5}$ spans $L_{4.6}$%
}\hbox{%
$|22\slashinfty22\rtimes D_{2}$ (shared)%
}%
}%
\hfill\copy\matricesbox
}%
\ifdim\wd\onelinebox>\myboxwidth
\hbox to \myboxwidth{%
$\Pi_{5,5}$ spans $L_{4.6}$%
\hfil
$|22\slashinfty22\rtimes D_{2}$ (shared)%
}%
\box\matricesbox
\else
\hbox to \myboxwidth{%
\unhbox\onelinebox
}%
\fi
\else
\hbox to \myboxwidth{%
$\Pi_{5,5}$ spans $L_{4.6}$%
\hfil}%
\hbox to \myboxwidth{%
$|22\slashinfty22\rtimes D_{2}$ (shared)%
\hfil}%
\box\matricesbox
\fi
}%
\hfill\discretionary{}{}{}%
\setbox\matricesbox=\hbox{%
{$\left[\!\llap{\phantom{%
\begingroup \smaller\smaller\smaller\begin{tabular}{@{}c@{}}%
\phantom{0}\\\phantom{0}\\\phantom{0}
\end{tabular}\endgroup%
}}\right.$}%
\begingroup \smaller\smaller\smaller\begin{tabular}{@{}c@{}}%
-1/3\\\phantom{0}\\\phantom{0}
\end{tabular}\endgroup%
\kern3pt%
\begingroup \smaller\smaller\smaller\begin{tabular}{@{}c@{}}%
\phantom{0}\\4/3\\\phantom{0}
\end{tabular}\endgroup%
\kern3pt%
\begingroup \smaller\smaller\smaller\begin{tabular}{@{}c@{}}%
\phantom{0}\\\phantom{0}\\2
\end{tabular}\endgroup%
{$\left.\llap{\phantom{%
\begingroup \smaller\smaller\smaller\begin{tabular}{@{}c@{}}%
\phantom{0}\\\phantom{0}\\\phantom{0}
\end{tabular}\endgroup%
}}\!\right]$}%
{$\left[\!\llap{\phantom{%
\begingroup \smaller\smaller\smaller\begin{tabular}{@{}c@{}}%
0\\0\\0
\end{tabular}\endgroup%
}}\right.$}%
\begingroup \smaller\smaller\smaller\begin{tabular}{@{}c@{}}%
1\\1\\0
\end{tabular}\endgroup%
\kern3pt%
\begingroup \smaller\smaller\smaller\begin{tabular}{@{}c@{}}%
4\\1\\2
\end{tabular}\endgroup%
\kern3pt%
\begingroup \smaller\smaller\smaller\begin{tabular}{@{}c@{}}%
2\\-1\\1
\end{tabular}\endgroup%
{$\left.\llap{\phantom{%
\begingroup \smaller\smaller\smaller\begin{tabular}{@{}c@{}}%
0\\0\\0
\end{tabular}\endgroup%
}}\!\right]$}%
}%
\ifdim\wd\matricesbox>\halfwidth\myboxwidth=\hsize\else\myboxwidth=\halfwidth\fi
\vbox{%
\ifdim\myboxwidth=\hsize
\setbox\onelinebox=\hbox{%
\vbox{\hbox{%
$\Pi_{5,6}$ spans $L_{141.7}$%
}\hbox{%
$|22\slashinfty22\rtimes D_{2}$%
}%
}%
\hfill\copy\matricesbox
}%
\ifdim\wd\onelinebox>\myboxwidth
\hbox to \myboxwidth{%
$\Pi_{5,6}$ spans $L_{141.7}$%
\hfil
$|22\slashinfty22\rtimes D_{2}$%
}%
\box\matricesbox
\else
\hbox to \myboxwidth{%
\unhbox\onelinebox
}%
\fi
\else
\hbox to \myboxwidth{%
$\Pi_{5,6}$ spans $L_{141.7}$%
\hfil}%
\hbox to \myboxwidth{%
$|22\slashinfty22\rtimes D_{2}$%
\hfil}%
\box\matricesbox
\fi
}%
\hfill\discretionary{}{}{}%
\setbox\matricesbox=\hbox{%
{$\left[\!\llap{\phantom{%
\begingroup \smaller\smaller\smaller\begin{tabular}{@{}c@{}}%
\phantom{0}\\\phantom{0}\\\phantom{0}
\end{tabular}\endgroup%
}}\right.$}%
\begingroup \smaller\smaller\smaller\begin{tabular}{@{}c@{}}%
-1/3\\\phantom{0}\\\phantom{0}
\end{tabular}\endgroup%
\kern3pt%
\begingroup \smaller\smaller\smaller\begin{tabular}{@{}c@{}}%
\phantom{0}\\1/3\\\phantom{0}
\end{tabular}\endgroup%
\kern3pt%
\begingroup \smaller\smaller\smaller\begin{tabular}{@{}c@{}}%
\phantom{0}\\\phantom{0}\\3
\end{tabular}\endgroup%
{$\left.\llap{\phantom{%
\begingroup \smaller\smaller\smaller\begin{tabular}{@{}c@{}}%
\phantom{0}\\\phantom{0}\\\phantom{0}
\end{tabular}\endgroup%
}}\!\right]$}%
{$\left[\!\llap{\phantom{%
\begingroup \smaller\smaller\smaller\begin{tabular}{@{}c@{}}%
0\\0\\0
\end{tabular}\endgroup%
}}\right.$}%
\begingroup \smaller\smaller\smaller\begin{tabular}{@{}c@{}}%
1\\2\\0
\end{tabular}\endgroup%
\kern3pt%
\begingroup \smaller\smaller\smaller\begin{tabular}{@{}c@{}}%
2\\1\\1
\end{tabular}\endgroup%
\kern3pt%
\begingroup \smaller\smaller\smaller\begin{tabular}{@{}c@{}}%
3\\-3\\1
\end{tabular}\endgroup%
{$\left.\llap{\phantom{%
\begingroup \smaller\smaller\smaller\begin{tabular}{@{}c@{}}%
0\\0\\0
\end{tabular}\endgroup%
}}\!\right]$}%
}%
\ifdim\wd\matricesbox>\halfwidth\myboxwidth=\hsize\else\myboxwidth=\halfwidth\fi
\vbox{%
\ifdim\myboxwidth=\hsize
\setbox\onelinebox=\hbox{%
\vbox{\hbox{%
$\Pi_{5,7}$ spans $L_{4.14}$%
}\hbox{%
$|22\slashinfty22\rtimes D_{2}$ (shared)%
}%
}%
\hfill\copy\matricesbox
}%
\ifdim\wd\onelinebox>\myboxwidth
\hbox to \myboxwidth{%
$\Pi_{5,7}$ spans $L_{4.14}$%
\hfil
$|22\slashinfty22\rtimes D_{2}$ (shared)%
}%
\box\matricesbox
\else
\hbox to \myboxwidth{%
\unhbox\onelinebox
}%
\fi
\else
\hbox to \myboxwidth{%
$\Pi_{5,7}$ spans $L_{4.14}$%
\hfil}%
\hbox to \myboxwidth{%
$|22\slashinfty22\rtimes D_{2}$ (shared)%
\hfil}%
\box\matricesbox
\fi
}%
\hfill\discretionary{}{}{}%
\setbox\matricesbox=\hbox{%
{$\left[\!\llap{\phantom{%
\begingroup \smaller\smaller\smaller\begin{tabular}{@{}c@{}}%
\phantom{0}\\\phantom{0}\\\phantom{0}
\end{tabular}\endgroup%
}}\right.$}%
\begingroup \smaller\smaller\smaller\begin{tabular}{@{}c@{}}%
-1/3\\\phantom{0}\\\phantom{0}
\end{tabular}\endgroup%
\kern3pt%
\begingroup \smaller\smaller\smaller\begin{tabular}{@{}c@{}}%
\phantom{0}\\5/6\\\phantom{0}
\end{tabular}\endgroup%
\kern3pt%
\begingroup \smaller\smaller\smaller\begin{tabular}{@{}c@{}}%
\phantom{0}\\\phantom{0}\\1/2
\end{tabular}\endgroup%
{$\left.\llap{\phantom{%
\begingroup \smaller\smaller\smaller\begin{tabular}{@{}c@{}}%
\phantom{0}\\\phantom{0}\\\phantom{0}
\end{tabular}\endgroup%
}}\!\right]$}%
{$\left[\!\llap{\phantom{%
\begingroup \smaller\smaller\smaller\begin{tabular}{@{}c@{}}%
0\\0\\0
\end{tabular}\endgroup%
}}\right.$}%
\begingroup \smaller\smaller\smaller\begin{tabular}{@{}c@{}}%
1\\-1\\-1
\end{tabular}\endgroup%
\kern3pt%
\begingroup \smaller\smaller\smaller\begin{tabular}{@{}c@{}}%
5\\1\\-5
\end{tabular}\endgroup%
\kern3pt%
\begingroup \smaller\smaller\smaller\begin{tabular}{@{}c@{}}%
5\\4\\0
\end{tabular}\endgroup%
{$\left.\llap{\phantom{%
\begingroup \smaller\smaller\smaller\begin{tabular}{@{}c@{}}%
0\\0\\0
\end{tabular}\endgroup%
}}\!\right]$}%
}%
\ifdim\wd\matricesbox>\halfwidth\myboxwidth=\hsize\else\myboxwidth=\halfwidth\fi
\vbox{%
\ifdim\myboxwidth=\hsize
\setbox\onelinebox=\hbox{%
\vbox{\hbox{%
$\Pi_{5,8}$ spans $L_{5.4}$%
}\hbox{%
$\slashtwo2\infty|\infty2\rtimes D_{2}$ (shared)%
}%
}%
\hfill\copy\matricesbox
}%
\ifdim\wd\onelinebox>\myboxwidth
\hbox to \myboxwidth{%
$\Pi_{5,8}$ spans $L_{5.4}$%
\hfil
$\slashtwo2\infty|\infty2\rtimes D_{2}$ (shared)%
}%
\box\matricesbox
\else
\hbox to \myboxwidth{%
\unhbox\onelinebox
}%
\fi
\else
\hbox to \myboxwidth{%
$\Pi_{5,8}$ spans $L_{5.4}$%
\hfil}%
\hbox to \myboxwidth{%
$\slashtwo2\infty|\infty2\rtimes D_{2}$ (shared)%
\hfil}%
\box\matricesbox
\fi
}%
\hfill\discretionary{}{}{}%
\setbox\matricesbox=\hbox{%
{$\left[\!\llap{\phantom{%
\begingroup \smaller\smaller\smaller\begin{tabular}{@{}c@{}}%
\phantom{0}\\\phantom{0}\\\phantom{0}
\end{tabular}\endgroup%
}}\right.$}%
\begingroup \smaller\smaller\smaller\begin{tabular}{@{}c@{}}%
-1/6\\\phantom{0}\\\phantom{0}
\end{tabular}\endgroup%
\kern3pt%
\begingroup \smaller\smaller\smaller\begin{tabular}{@{}c@{}}%
\phantom{0}\\1/6\\\phantom{0}
\end{tabular}\endgroup%
\kern3pt%
\begingroup \smaller\smaller\smaller\begin{tabular}{@{}c@{}}%
\phantom{0}\\\phantom{0}\\5/2
\end{tabular}\endgroup%
{$\left.\llap{\phantom{%
\begingroup \smaller\smaller\smaller\begin{tabular}{@{}c@{}}%
\phantom{0}\\\phantom{0}\\\phantom{0}
\end{tabular}\endgroup%
}}\!\right]$}%
{$\left[\!\llap{\phantom{%
\begingroup \smaller\smaller\smaller\begin{tabular}{@{}c@{}}%
0\\0\\0
\end{tabular}\endgroup%
}}\right.$}%
\begingroup \smaller\smaller\smaller\begin{tabular}{@{}c@{}}%
2\\4\\0
\end{tabular}\endgroup%
\kern3pt%
\begingroup \smaller\smaller\smaller\begin{tabular}{@{}c@{}}%
2\\1\\-1
\end{tabular}\endgroup%
\kern3pt%
\begingroup \smaller\smaller\smaller\begin{tabular}{@{}c@{}}%
10\\-10\\-2
\end{tabular}\endgroup%
{$\left.\llap{\phantom{%
\begingroup \smaller\smaller\smaller\begin{tabular}{@{}c@{}}%
0\\0\\0
\end{tabular}\endgroup%
}}\!\right]$}%
}%
\ifdim\wd\matricesbox>\halfwidth\myboxwidth=\hsize\else\myboxwidth=\halfwidth\fi
\vbox{%
\ifdim\myboxwidth=\hsize
\setbox\onelinebox=\hbox{%
\vbox{\hbox{%
$\Pi_{5,9}$ spans $L_{10.2}$%
}\hbox{%
$|22\slashinfty22\rtimes D_{2}$%
}%
}%
\hfill\copy\matricesbox
}%
\ifdim\wd\onelinebox>\myboxwidth
\hbox to \myboxwidth{%
$\Pi_{5,9}$ spans $L_{10.2}$%
\hfil
$|22\slashinfty22\rtimes D_{2}$%
}%
\box\matricesbox
\else
\hbox to \myboxwidth{%
\unhbox\onelinebox
}%
\fi
\else
\hbox to \myboxwidth{%
$\Pi_{5,9}$ spans $L_{10.2}$%
\hfil}%
\hbox to \myboxwidth{%
$|22\slashinfty22\rtimes D_{2}$%
\hfil}%
\box\matricesbox
\fi
}%
\hfill\discretionary{}{}{}%
\setbox\matricesbox=\hbox{%
{$\left[\!\llap{\phantom{%
\begingroup \smaller\smaller\smaller\begin{tabular}{@{}c@{}}%
\phantom{0}\\\phantom{0}\\\phantom{0}
\end{tabular}\endgroup%
}}\right.$}%
\begingroup \smaller\smaller\smaller\begin{tabular}{@{}c@{}}%
-1/7\\\phantom{0}\\\phantom{0}
\end{tabular}\endgroup%
\kern3pt%
\begingroup \smaller\smaller\smaller\begin{tabular}{@{}c@{}}%
\phantom{0}\\15/14\\\phantom{0}
\end{tabular}\endgroup%
\kern3pt%
\begingroup \smaller\smaller\smaller\begin{tabular}{@{}c@{}}%
\phantom{0}\\\phantom{0}\\3/2
\end{tabular}\endgroup%
{$\left.\llap{\phantom{%
\begingroup \smaller\smaller\smaller\begin{tabular}{@{}c@{}}%
\phantom{0}\\\phantom{0}\\\phantom{0}
\end{tabular}\endgroup%
}}\!\right]$}%
{$\left[\!\llap{\phantom{%
\begingroup \smaller\smaller\smaller\begin{tabular}{@{}c@{}}%
0\\0\\0
\end{tabular}\endgroup%
}}\right.$}%
\begingroup \smaller\smaller\smaller\begin{tabular}{@{}c@{}}%
3\\-2\\0
\end{tabular}\endgroup%
\kern3pt%
\begingroup \smaller\smaller\smaller\begin{tabular}{@{}c@{}}%
15\\-3\\-5
\end{tabular}\endgroup%
\kern3pt%
\begingroup \smaller\smaller\smaller\begin{tabular}{@{}c@{}}%
2\\1\\-1
\end{tabular}\endgroup%
{$\left.\llap{\phantom{%
\begingroup \smaller\smaller\smaller\begin{tabular}{@{}c@{}}%
0\\0\\0
\end{tabular}\endgroup%
}}\!\right]$}%
}%
\ifdim\wd\matricesbox>\halfwidth\myboxwidth=\hsize\else\myboxwidth=\halfwidth\fi
\vbox{%
\ifdim\myboxwidth=\hsize
\setbox\onelinebox=\hbox{%
\vbox{\hbox{%
$\Pi_{5,10}$ spans $L_{16.5}$%
}\hbox{%
$|22\slashthree22\rtimes D_{2}$%
}%
}%
\hfill\copy\matricesbox
}%
\ifdim\wd\onelinebox>\myboxwidth
\hbox to \myboxwidth{%
$\Pi_{5,10}$ spans $L_{16.5}$%
\hfil
$|22\slashthree22\rtimes D_{2}$%
}%
\box\matricesbox
\else
\hbox to \myboxwidth{%
\unhbox\onelinebox
}%
\fi
\else
\hbox to \myboxwidth{%
$\Pi_{5,10}$ spans $L_{16.5}$%
\hfil}%
\hbox to \myboxwidth{%
$|22\slashthree22\rtimes D_{2}$%
\hfil}%
\box\matricesbox
\fi
}%
\hfill\discretionary{}{}{}%
\setbox\matricesbox=\hbox{%
{$\left[\!\llap{\phantom{%
\begingroup \smaller\smaller\smaller\begin{tabular}{@{}c@{}}%
\phantom{0}\\\phantom{0}\\\phantom{0}
\end{tabular}\endgroup%
}}\right.$}%
\begingroup \smaller\smaller\smaller\begin{tabular}{@{}c@{}}%
-7/11\\\phantom{0}\\\phantom{0}
\end{tabular}\endgroup%
\kern3pt%
\begingroup \smaller\smaller\smaller\begin{tabular}{@{}c@{}}%
\phantom{0}\\1/22\\\phantom{0}
\end{tabular}\endgroup%
\kern3pt%
\begingroup \smaller\smaller\smaller\begin{tabular}{@{}c@{}}%
\phantom{0}\\\phantom{0}\\1/2
\end{tabular}\endgroup%
{$\left.\llap{\phantom{%
\begingroup \smaller\smaller\smaller\begin{tabular}{@{}c@{}}%
\phantom{0}\\\phantom{0}\\\phantom{0}
\end{tabular}\endgroup%
}}\!\right]$}%
{$\left[\!\llap{\phantom{%
\begingroup \smaller\smaller\smaller\begin{tabular}{@{}c@{}}%
0\\0\\0
\end{tabular}\endgroup%
}}\right.$}%
\begingroup \smaller\smaller\smaller\begin{tabular}{@{}c@{}}%
1\\-5\\-1
\end{tabular}\endgroup%
\kern3pt%
\begingroup \smaller\smaller\smaller\begin{tabular}{@{}c@{}}%
2\\1\\-3
\end{tabular}\endgroup%
\kern3pt%
\begingroup \smaller\smaller\smaller\begin{tabular}{@{}c@{}}%
1\\6\\0
\end{tabular}\endgroup%
{$\left.\llap{\phantom{%
\begingroup \smaller\smaller\smaller\begin{tabular}{@{}c@{}}%
0\\0\\0
\end{tabular}\endgroup%
}}\!\right]$}%
}%
\ifdim\wd\matricesbox>\halfwidth\myboxwidth=\hsize\else\myboxwidth=\halfwidth\fi
\vbox{%
\ifdim\myboxwidth=\hsize
\setbox\onelinebox=\hbox{%
\vbox{\hbox{%
$\Pi_{5,11}$ spans $L_{8.2}$%
}\hbox{%
$\slashtwo24|42\rtimes D_{2}$%
}%
}%
\hfill\copy\matricesbox
}%
\ifdim\wd\onelinebox>\myboxwidth
\hbox to \myboxwidth{%
$\Pi_{5,11}$ spans $L_{8.2}$%
\hfil
$\slashtwo24|42\rtimes D_{2}$%
}%
\box\matricesbox
\else
\hbox to \myboxwidth{%
\unhbox\onelinebox
}%
\fi
\else
\hbox to \myboxwidth{%
$\Pi_{5,11}$ spans $L_{8.2}$%
\hfil}%
\hbox to \myboxwidth{%
$\slashtwo24|42\rtimes D_{2}$%
\hfil}%
\box\matricesbox
\fi
}%
\hfill\discretionary{}{}{}%
\setbox\matricesbox=\hbox{%
{$\left[\!\llap{\phantom{%
\begingroup \smaller\smaller\smaller\begin{tabular}{@{}c@{}}%
\phantom{0}\\\phantom{0}\\\phantom{0}
\end{tabular}\endgroup%
}}\right.$}%
\begingroup \smaller\smaller\smaller\begin{tabular}{@{}c@{}}%
-1/10\\\phantom{0}\\\phantom{0}
\end{tabular}\endgroup%
\kern3pt%
\begingroup \smaller\smaller\smaller\begin{tabular}{@{}c@{}}%
\phantom{0}\\18/5\\\phantom{0}
\end{tabular}\endgroup%
\kern3pt%
\begingroup \smaller\smaller\smaller\begin{tabular}{@{}c@{}}%
\phantom{0}\\\phantom{0}\\3/2
\end{tabular}\endgroup%
{$\left.\llap{\phantom{%
\begingroup \smaller\smaller\smaller\begin{tabular}{@{}c@{}}%
\phantom{0}\\\phantom{0}\\\phantom{0}
\end{tabular}\endgroup%
}}\!\right]$}%
{$\left[\!\llap{\phantom{%
\begingroup \smaller\smaller\smaller\begin{tabular}{@{}c@{}}%
0\\0\\0
\end{tabular}\endgroup%
}}\right.$}%
\begingroup \smaller\smaller\smaller\begin{tabular}{@{}c@{}}%
8\\2\\0
\end{tabular}\endgroup%
\kern3pt%
\begingroup \smaller\smaller\smaller\begin{tabular}{@{}c@{}}%
9\\1\\3
\end{tabular}\endgroup%
\kern3pt%
\begingroup \smaller\smaller\smaller\begin{tabular}{@{}c@{}}%
6\\-1\\2
\end{tabular}\endgroup%
{$\left.\llap{\phantom{%
\begingroup \smaller\smaller\smaller\begin{tabular}{@{}c@{}}%
0\\0\\0
\end{tabular}\endgroup%
}}\!\right]$}%
}%
\ifdim\wd\matricesbox>\halfwidth\myboxwidth=\hsize\else\myboxwidth=\halfwidth\fi
\vbox{%
\ifdim\myboxwidth=\hsize
\setbox\onelinebox=\hbox{%
\vbox{\hbox{%
$\Pi_{5,12}$ spans $L_{4.11}$%
}\hbox{%
$|22\slashinfty22\rtimes D_{2}$ (shared)%
}%
}%
\hfill\copy\matricesbox
}%
\ifdim\wd\onelinebox>\myboxwidth
\hbox to \myboxwidth{%
$\Pi_{5,12}$ spans $L_{4.11}$%
\hfil
$|22\slashinfty22\rtimes D_{2}$ (shared)%
}%
\box\matricesbox
\else
\hbox to \myboxwidth{%
\unhbox\onelinebox
}%
\fi
\else
\hbox to \myboxwidth{%
$\Pi_{5,12}$ spans $L_{4.11}$%
\hfil}%
\hbox to \myboxwidth{%
$|22\slashinfty22\rtimes D_{2}$ (shared)%
\hfil}%
\box\matricesbox
\fi
}%
\hfill\discretionary{}{}{}%
\setbox\matricesbox=\hbox{%
{$\left[\!\llap{\phantom{%
\begingroup \smaller\smaller\smaller\begin{tabular}{@{}c@{}}%
\phantom{0}\\\phantom{0}\\\phantom{0}
\end{tabular}\endgroup%
}}\right.$}%
\begingroup \smaller\smaller\smaller\begin{tabular}{@{}c@{}}%
-1/14\\\phantom{0}\\\phantom{0}
\end{tabular}\endgroup%
\kern3pt%
\begingroup \smaller\smaller\smaller\begin{tabular}{@{}c@{}}%
\phantom{0}\\15/14\\\phantom{0}
\end{tabular}\endgroup%
\kern3pt%
\begingroup \smaller\smaller\smaller\begin{tabular}{@{}c@{}}%
\phantom{0}\\\phantom{0}\\15/2
\end{tabular}\endgroup%
{$\left.\llap{\phantom{%
\begingroup \smaller\smaller\smaller\begin{tabular}{@{}c@{}}%
\phantom{0}\\\phantom{0}\\\phantom{0}
\end{tabular}\endgroup%
}}\!\right]$}%
{$\left[\!\llap{\phantom{%
\begingroup \smaller\smaller\smaller\begin{tabular}{@{}c@{}}%
0\\0\\0
\end{tabular}\endgroup%
}}\right.$}%
\begingroup \smaller\smaller\smaller\begin{tabular}{@{}c@{}}%
10\\-4\\0
\end{tabular}\endgroup%
\kern3pt%
\begingroup \smaller\smaller\smaller\begin{tabular}{@{}c@{}}%
6\\-1\\1
\end{tabular}\endgroup%
\kern3pt%
\begingroup \smaller\smaller\smaller\begin{tabular}{@{}c@{}}%
10\\3\\1
\end{tabular}\endgroup%
{$\left.\llap{\phantom{%
\begingroup \smaller\smaller\smaller\begin{tabular}{@{}c@{}}%
0\\0\\0
\end{tabular}\endgroup%
}}\!\right]$}%
}%
\ifdim\wd\matricesbox>\halfwidth\myboxwidth=\hsize\else\myboxwidth=\halfwidth\fi
\vbox{%
\ifdim\myboxwidth=\hsize
\setbox\onelinebox=\hbox{%
\vbox{\hbox{%
$\Pi_{5,13}$ spans $L_{31.9}$%
}\hbox{%
$|22\slashthree22\rtimes D_{2}$ (shared)%
}%
}%
\hfill\copy\matricesbox
}%
\ifdim\wd\onelinebox>\myboxwidth
\hbox to \myboxwidth{%
$\Pi_{5,13}$ spans $L_{31.9}$%
\hfil
$|22\slashthree22\rtimes D_{2}$ (shared)%
}%
\box\matricesbox
\else
\hbox to \myboxwidth{%
\unhbox\onelinebox
}%
\fi
\else
\hbox to \myboxwidth{%
$\Pi_{5,13}$ spans $L_{31.9}$%
\hfil}%
\hbox to \myboxwidth{%
$|22\slashthree22\rtimes D_{2}$ (shared)%
\hfil}%
\box\matricesbox
\fi
}%
\hfill\discretionary{}{}{}%
\setbox\matricesbox=\hbox{%
{$\left[\!\llap{\phantom{%
\begingroup \smaller\smaller\smaller\begin{tabular}{@{}c@{}}%
\phantom{0}\\\phantom{0}\\\phantom{0}
\end{tabular}\endgroup%
}}\right.$}%
\begingroup \smaller\smaller\smaller\begin{tabular}{@{}c@{}}%
-1/40\\\phantom{0}\\\phantom{0}
\end{tabular}\endgroup%
\kern3pt%
\begingroup \smaller\smaller\smaller\begin{tabular}{@{}c@{}}%
\phantom{0}\\39/10\\\phantom{0}
\end{tabular}\endgroup%
\kern3pt%
\begingroup \smaller\smaller\smaller\begin{tabular}{@{}c@{}}%
\phantom{0}\\\phantom{0}\\39
\end{tabular}\endgroup%
{$\left.\llap{\phantom{%
\begingroup \smaller\smaller\smaller\begin{tabular}{@{}c@{}}%
\phantom{0}\\\phantom{0}\\\phantom{0}
\end{tabular}\endgroup%
}}\!\right]$}%
{$\left[\!\llap{\phantom{%
\begingroup \smaller\smaller\smaller\begin{tabular}{@{}c@{}}%
0\\0\\0
\end{tabular}\endgroup%
}}\right.$}%
\begingroup \smaller\smaller\smaller\begin{tabular}{@{}c@{}}%
12\\-2\\0
\end{tabular}\endgroup%
\kern3pt%
\begingroup \smaller\smaller\smaller\begin{tabular}{@{}c@{}}%
26\\-1\\-1
\end{tabular}\endgroup%
\kern3pt%
\begingroup \smaller\smaller\smaller\begin{tabular}{@{}c@{}}%
78\\7\\-1
\end{tabular}\endgroup%
{$\left.\llap{\phantom{%
\begingroup \smaller\smaller\smaller\begin{tabular}{@{}c@{}}%
0\\0\\0
\end{tabular}\endgroup%
}}\!\right]$}%
}%
\ifdim\wd\matricesbox>\halfwidth\myboxwidth=\hsize\else\myboxwidth=\halfwidth\fi
\vbox{%
\ifdim\myboxwidth=\hsize
\setbox\onelinebox=\hbox{%
\vbox{\hbox{%
$\Pi_{5,14}$ spans $L_{41.10}$%
}\hbox{%
$2|26\slashtwo6\rtimes D_{2}$%
}%
}%
\hfill\copy\matricesbox
}%
\ifdim\wd\onelinebox>\myboxwidth
\hbox to \myboxwidth{%
$\Pi_{5,14}$ spans $L_{41.10}$%
\hfil
$2|26\slashtwo6\rtimes D_{2}$%
}%
\box\matricesbox
\else
\hbox to \myboxwidth{%
\unhbox\onelinebox
}%
\fi
\else
\hbox to \myboxwidth{%
$\Pi_{5,14}$ spans $L_{41.10}$%
\hfil}%
\hbox to \myboxwidth{%
$2|26\slashtwo6\rtimes D_{2}$%
\hfil}%
\box\matricesbox
\fi
}%
\hfill\discretionary{}{}{}%
\setbox\matricesbox=\hbox{%
{$\left[\!\llap{\phantom{%
\begingroup \smaller\smaller\smaller\begin{tabular}{@{}c@{}}%
\phantom{0}\\\phantom{0}\\\phantom{0}
\end{tabular}\endgroup%
}}\right.$}%
\begingroup \smaller\smaller\smaller\begin{tabular}{@{}c@{}}%
-1/4\\\phantom{0}\\\phantom{0}
\end{tabular}\endgroup%
\kern3pt%
\begingroup \smaller\smaller\smaller\begin{tabular}{@{}c@{}}%
\phantom{0}\\3\\\phantom{0}
\end{tabular}\endgroup%
\kern3pt%
\begingroup \smaller\smaller\smaller\begin{tabular}{@{}c@{}}%
\phantom{0}\\\phantom{0}\\3
\end{tabular}\endgroup%
{$\left.\llap{\phantom{%
\begingroup \smaller\smaller\smaller\begin{tabular}{@{}c@{}}%
\phantom{0}\\\phantom{0}\\\phantom{0}
\end{tabular}\endgroup%
}}\!\right]$}%
{$\left[\!\llap{\phantom{%
\begingroup \smaller\smaller\smaller\begin{tabular}{@{}c@{}}%
0\\0\\0
\end{tabular}\endgroup%
}}\right.$}%
\begingroup \smaller\smaller\smaller\begin{tabular}{@{}c@{}}%
6\\2\\-1
\end{tabular}\endgroup%
\kern3pt%
\begingroup \smaller\smaller\smaller\begin{tabular}{@{}c@{}}%
2\\0\\-1
\end{tabular}\endgroup%
\kern3pt%
\begingroup \smaller\smaller\smaller\begin{tabular}{@{}c@{}}%
2\\-1\\0
\end{tabular}\endgroup%
{$\left.\llap{\phantom{%
\begingroup \smaller\smaller\smaller\begin{tabular}{@{}c@{}}%
0\\0\\0
\end{tabular}\endgroup%
}}\!\right]$}%
}%
\ifdim\wd\matricesbox>\halfwidth\myboxwidth=\hsize\else\myboxwidth=\halfwidth\fi
\vbox{%
\ifdim\myboxwidth=\hsize
\setbox\onelinebox=\hbox{%
\vbox{\hbox{%
$\Pi_{5,15}$ spans $L_{3.4}$%
}\hbox{%
$\slashtwo23|32\rtimes D_{2}$%
}%
}%
\hfill\copy\matricesbox
}%
\ifdim\wd\onelinebox>\myboxwidth
\hbox to \myboxwidth{%
$\Pi_{5,15}$ spans $L_{3.4}$%
\hfil
$\slashtwo23|32\rtimes D_{2}$%
}%
\box\matricesbox
\else
\hbox to \myboxwidth{%
\unhbox\onelinebox
}%
\fi
\else
\hbox to \myboxwidth{%
$\Pi_{5,15}$ spans $L_{3.4}$%
\hfil}%
\hbox to \myboxwidth{%
$\slashtwo23|32\rtimes D_{2}$%
\hfil}%
\box\matricesbox
\fi
}%
\hfill\discretionary{}{}{}%
\setbox\matricesbox=\hbox{%
{$\left[\!\llap{\phantom{%
\begingroup \smaller\smaller\smaller\begin{tabular}{@{}c@{}}%
\phantom{0}\\\phantom{0}\\\phantom{0}
\end{tabular}\endgroup%
}}\right.$}%
\begingroup \smaller\smaller\smaller\begin{tabular}{@{}c@{}}%
-1/7\\\phantom{0}\\\phantom{0}
\end{tabular}\endgroup%
\kern3pt%
\begingroup \smaller\smaller\smaller\begin{tabular}{@{}c@{}}%
\phantom{0}\\9/7\\\phantom{0}
\end{tabular}\endgroup%
\kern3pt%
\begingroup \smaller\smaller\smaller\begin{tabular}{@{}c@{}}%
\phantom{0}\\\phantom{0}\\3
\end{tabular}\endgroup%
{$\left.\llap{\phantom{%
\begingroup \smaller\smaller\smaller\begin{tabular}{@{}c@{}}%
\phantom{0}\\\phantom{0}\\\phantom{0}
\end{tabular}\endgroup%
}}\!\right]$}%
{$\left[\!\llap{\phantom{%
\begingroup \smaller\smaller\smaller\begin{tabular}{@{}c@{}}%
0\\0\\0
\end{tabular}\endgroup%
}}\right.$}%
\begingroup \smaller\smaller\smaller\begin{tabular}{@{}c@{}}%
9\\4\\0
\end{tabular}\endgroup%
\kern3pt%
\begingroup \smaller\smaller\smaller\begin{tabular}{@{}c@{}}%
8\\2\\2
\end{tabular}\endgroup%
\kern3pt%
\begingroup \smaller\smaller\smaller\begin{tabular}{@{}c@{}}%
3\\-1\\1
\end{tabular}\endgroup%
{$\left.\llap{\phantom{%
\begingroup \smaller\smaller\smaller\begin{tabular}{@{}c@{}}%
0\\0\\0
\end{tabular}\endgroup%
}}\!\right]$}%
}%
\ifdim\wd\matricesbox>\halfwidth\myboxwidth=\hsize\else\myboxwidth=\halfwidth\fi
\vbox{%
\ifdim\myboxwidth=\hsize
\setbox\onelinebox=\hbox{%
\vbox{\hbox{%
$\Pi_{5,16}$ spans $L_{4.25}$%
}\hbox{%
$|22\slashinfty22\rtimes D_{2}$ (shared)%
}%
}%
\hfill\copy\matricesbox
}%
\ifdim\wd\onelinebox>\myboxwidth
\hbox to \myboxwidth{%
$\Pi_{5,16}$ spans $L_{4.25}$%
\hfil
$|22\slashinfty22\rtimes D_{2}$ (shared)%
}%
\box\matricesbox
\else
\hbox to \myboxwidth{%
\unhbox\onelinebox
}%
\fi
\else
\hbox to \myboxwidth{%
$\Pi_{5,16}$ spans $L_{4.25}$%
\hfil}%
\hbox to \myboxwidth{%
$|22\slashinfty22\rtimes D_{2}$ (shared)%
\hfil}%
\box\matricesbox
\fi
}%
\hfill\discretionary{}{}{}%
\setbox\matricesbox=\hbox{%
{$\left[\!\llap{\phantom{%
\begingroup \smaller\smaller\smaller\begin{tabular}{@{}c@{}}%
\phantom{0}\\\phantom{0}\\\phantom{0}
\end{tabular}\endgroup%
}}\right.$}%
\begingroup \smaller\smaller\smaller\begin{tabular}{@{}c@{}}%
-1/16\\\phantom{0}\\\phantom{0}
\end{tabular}\endgroup%
\kern3pt%
\begingroup \smaller\smaller\smaller\begin{tabular}{@{}c@{}}%
\phantom{0}\\15/4\\\phantom{0}
\end{tabular}\endgroup%
\kern3pt%
\begingroup \smaller\smaller\smaller\begin{tabular}{@{}c@{}}%
\phantom{0}\\\phantom{0}\\1/2
\end{tabular}\endgroup%
{$\left.\llap{\phantom{%
\begingroup \smaller\smaller\smaller\begin{tabular}{@{}c@{}}%
\phantom{0}\\\phantom{0}\\\phantom{0}
\end{tabular}\endgroup%
}}\!\right]$}%
{$\left[\!\llap{\phantom{%
\begingroup \smaller\smaller\smaller\begin{tabular}{@{}c@{}}%
0\\0\\0
\end{tabular}\endgroup%
}}\right.$}%
\begingroup \smaller\smaller\smaller\begin{tabular}{@{}c@{}}%
24\\4\\0
\end{tabular}\endgroup%
\kern3pt%
\begingroup \smaller\smaller\smaller\begin{tabular}{@{}c@{}}%
10\\1\\-5
\end{tabular}\endgroup%
\kern3pt%
\begingroup \smaller\smaller\smaller\begin{tabular}{@{}c@{}}%
6\\-1\\-3
\end{tabular}\endgroup%
{$\left.\llap{\phantom{%
\begingroup \smaller\smaller\smaller\begin{tabular}{@{}c@{}}%
0\\0\\0
\end{tabular}\endgroup%
}}\!\right]$}%
}%
\ifdim\wd\matricesbox>\halfwidth\myboxwidth=\hsize\else\myboxwidth=\halfwidth\fi
\vbox{%
\ifdim\myboxwidth=\hsize
\setbox\onelinebox=\hbox{%
\vbox{\hbox{%
$\Pi_{5,17}$ spans $L_{31.1}$%
}\hbox{%
$|22\slashthree22\rtimes D_{2}$ (shared)%
}%
}%
\hfill\copy\matricesbox
}%
\ifdim\wd\onelinebox>\myboxwidth
\hbox to \myboxwidth{%
$\Pi_{5,17}$ spans $L_{31.1}$%
\hfil
$|22\slashthree22\rtimes D_{2}$ (shared)%
}%
\box\matricesbox
\else
\hbox to \myboxwidth{%
\unhbox\onelinebox
}%
\fi
\else
\hbox to \myboxwidth{%
$\Pi_{5,17}$ spans $L_{31.1}$%
\hfil}%
\hbox to \myboxwidth{%
$|22\slashthree22\rtimes D_{2}$ (shared)%
\hfil}%
\box\matricesbox
\fi
}%
\hfill\discretionary{}{}{}%
\setbox\matricesbox=\hbox{%
{$\left[\!\llap{\phantom{%
\begingroup \smaller\smaller\smaller\begin{tabular}{@{}c@{}}%
\phantom{0}\\\phantom{0}\\\phantom{0}
\end{tabular}\endgroup%
}}\right.$}%
\begingroup \smaller\smaller\smaller\begin{tabular}{@{}c@{}}%
-1/28\\\phantom{0}\\\phantom{0}
\end{tabular}\endgroup%
\kern3pt%
\begingroup \smaller\smaller\smaller\begin{tabular}{@{}c@{}}%
\phantom{0}\\15/7\\\phantom{0}
\end{tabular}\endgroup%
\kern3pt%
\begingroup \smaller\smaller\smaller\begin{tabular}{@{}c@{}}%
\phantom{0}\\\phantom{0}\\5
\end{tabular}\endgroup%
{$\left.\llap{\phantom{%
\begingroup \smaller\smaller\smaller\begin{tabular}{@{}c@{}}%
\phantom{0}\\\phantom{0}\\\phantom{0}
\end{tabular}\endgroup%
}}\!\right]$}%
{$\left[\!\llap{\phantom{%
\begingroup \smaller\smaller\smaller\begin{tabular}{@{}c@{}}%
0\\0\\0
\end{tabular}\endgroup%
}}\right.$}%
\begingroup \smaller\smaller\smaller\begin{tabular}{@{}c@{}}%
2\\-1\\0
\end{tabular}\endgroup%
\kern3pt%
\begingroup \smaller\smaller\smaller\begin{tabular}{@{}c@{}}%
60\\-2\\-6
\end{tabular}\endgroup%
\kern3pt%
\begingroup \smaller\smaller\smaller\begin{tabular}{@{}c@{}}%
10\\2\\-1
\end{tabular}\endgroup%
{$\left.\llap{\phantom{%
\begingroup \smaller\smaller\smaller\begin{tabular}{@{}c@{}}%
0\\0\\0
\end{tabular}\endgroup%
}}\!\right]$}%
}%
\ifdim\wd\matricesbox>\halfwidth\myboxwidth=\hsize\else\myboxwidth=\halfwidth\fi
\vbox{%
\ifdim\myboxwidth=\hsize
\setbox\onelinebox=\hbox{%
\vbox{\hbox{%
$\Pi_{5,18}$ spans $L_{19.8}$%
}\hbox{%
$|22\slashtwo22\rtimes D_{2}$%
}%
}%
\hfill\copy\matricesbox
}%
\ifdim\wd\onelinebox>\myboxwidth
\hbox to \myboxwidth{%
$\Pi_{5,18}$ spans $L_{19.8}$%
\hfil
$|22\slashtwo22\rtimes D_{2}$%
}%
\box\matricesbox
\else
\hbox to \myboxwidth{%
\unhbox\onelinebox
}%
\fi
\else
\hbox to \myboxwidth{%
$\Pi_{5,18}$ spans $L_{19.8}$%
\hfil}%
\hbox to \myboxwidth{%
$|22\slashtwo22\rtimes D_{2}$%
\hfil}%
\box\matricesbox
\fi
}%
\hfill\discretionary{}{}{}%
\setbox\matricesbox=\hbox{%
{$\left[\!\llap{\phantom{%
\begingroup \smaller\smaller\smaller\begin{tabular}{@{}c@{}}%
\phantom{0}\\\phantom{0}\\\phantom{0}
\end{tabular}\endgroup%
}}\right.$}%
\begingroup \smaller\smaller\smaller\begin{tabular}{@{}c@{}}%
-1/10\\\phantom{0}\\\phantom{0}
\end{tabular}\endgroup%
\kern3pt%
\begingroup \smaller\smaller\smaller\begin{tabular}{@{}c@{}}%
\phantom{0}\\3/5\\\phantom{0}
\end{tabular}\endgroup%
\kern3pt%
\begingroup \smaller\smaller\smaller\begin{tabular}{@{}c@{}}%
\phantom{0}\\\phantom{0}\\1/2
\end{tabular}\endgroup%
{$\left.\llap{\phantom{%
\begingroup \smaller\smaller\smaller\begin{tabular}{@{}c@{}}%
\phantom{0}\\\phantom{0}\\\phantom{0}
\end{tabular}\endgroup%
}}\!\right]$}%
{$\left[\!\llap{\phantom{%
\begingroup \smaller\smaller\smaller\begin{tabular}{@{}c@{}}%
0\\0\\0
\end{tabular}\endgroup%
}}\right.$}%
\begingroup \smaller\smaller\smaller\begin{tabular}{@{}c@{}}%
6\\4\\0
\end{tabular}\endgroup%
\kern3pt%
\begingroup \smaller\smaller\smaller\begin{tabular}{@{}c@{}}%
16\\4\\-8
\end{tabular}\endgroup%
\kern3pt%
\begingroup \smaller\smaller\smaller\begin{tabular}{@{}c@{}}%
1\\-1\\-1
\end{tabular}\endgroup%
{$\left.\llap{\phantom{%
\begingroup \smaller\smaller\smaller\begin{tabular}{@{}c@{}}%
0\\0\\0
\end{tabular}\endgroup%
}}\!\right]$}%
}%
\ifdim\wd\matricesbox>\halfwidth\myboxwidth=\hsize\else\myboxwidth=\halfwidth\fi
\vbox{%
\ifdim\myboxwidth=\hsize
\setbox\onelinebox=\hbox{%
\vbox{\hbox{%
$\Pi_{5,19}$ spans $L_{123.4}$%
}\hbox{%
$22|22\slashtwo\rtimes D_{2}$ (shared)%
}%
}%
\hfill\copy\matricesbox
}%
\ifdim\wd\onelinebox>\myboxwidth
\hbox to \myboxwidth{%
$\Pi_{5,19}$ spans $L_{123.4}$%
\hfil
$22|22\slashtwo\rtimes D_{2}$ (shared)%
}%
\box\matricesbox
\else
\hbox to \myboxwidth{%
\unhbox\onelinebox
}%
\fi
\else
\hbox to \myboxwidth{%
$\Pi_{5,19}$ spans $L_{123.4}$%
\hfil}%
\hbox to \myboxwidth{%
$22|22\slashtwo\rtimes D_{2}$ (shared)%
\hfil}%
\box\matricesbox
\fi
}%
\hfill\discretionary{}{}{}%
\setbox\matricesbox=\hbox{%
{$\left[\!\llap{\phantom{%
\begingroup \smaller\smaller\smaller\begin{tabular}{@{}c@{}}%
\phantom{0}\\\phantom{0}\\\phantom{0}
\end{tabular}\endgroup%
}}\right.$}%
\begingroup \smaller\smaller\smaller\begin{tabular}{@{}c@{}}%
-1/11\\\phantom{0}\\\phantom{0}
\end{tabular}\endgroup%
\kern3pt%
\begingroup \smaller\smaller\smaller\begin{tabular}{@{}c@{}}%
\phantom{0}\\3/11\\\phantom{0}
\end{tabular}\endgroup%
\kern3pt%
\begingroup \smaller\smaller\smaller\begin{tabular}{@{}c@{}}%
\phantom{0}\\\phantom{0}\\1
\end{tabular}\endgroup%
{$\left.\llap{\phantom{%
\begingroup \smaller\smaller\smaller\begin{tabular}{@{}c@{}}%
\phantom{0}\\\phantom{0}\\\phantom{0}
\end{tabular}\endgroup%
}}\!\right]$}%
{$\left[\!\llap{\phantom{%
\begingroup \smaller\smaller\smaller\begin{tabular}{@{}c@{}}%
0\\0\\0
\end{tabular}\endgroup%
}}\right.$}%
\begingroup \smaller\smaller\smaller\begin{tabular}{@{}c@{}}%
1\\2\\0
\end{tabular}\endgroup%
\kern3pt%
\begingroup \smaller\smaller\smaller\begin{tabular}{@{}c@{}}%
6\\1\\3
\end{tabular}\endgroup%
\kern3pt%
\begingroup \smaller\smaller\smaller\begin{tabular}{@{}c@{}}%
8\\-6\\2
\end{tabular}\endgroup%
{$\left.\llap{\phantom{%
\begingroup \smaller\smaller\smaller\begin{tabular}{@{}c@{}}%
0\\0\\0
\end{tabular}\endgroup%
}}\!\right]$}%
}%
\ifdim\wd\matricesbox>\halfwidth\myboxwidth=\hsize\else\myboxwidth=\halfwidth\fi
\vbox{%
\ifdim\myboxwidth=\hsize
\setbox\onelinebox=\hbox{%
\vbox{\hbox{%
$\Pi_{5,20}$ spans $L_{123.5}$%
}\hbox{%
$|22\slashtwo22\rtimes D_{2}$ (shared)%
}%
}%
\hfill\copy\matricesbox
}%
\ifdim\wd\onelinebox>\myboxwidth
\hbox to \myboxwidth{%
$\Pi_{5,20}$ spans $L_{123.5}$%
\hfil
$|22\slashtwo22\rtimes D_{2}$ (shared)%
}%
\box\matricesbox
\else
\hbox to \myboxwidth{%
\unhbox\onelinebox
}%
\fi
\else
\hbox to \myboxwidth{%
$\Pi_{5,20}$ spans $L_{123.5}$%
\hfil}%
\hbox to \myboxwidth{%
$|22\slashtwo22\rtimes D_{2}$ (shared)%
\hfil}%
\box\matricesbox
\fi
}%
\hfill\discretionary{}{}{}%
\setbox\matricesbox=\hbox{%
{$\left[\!\llap{\phantom{%
\begingroup \smaller\smaller\smaller\begin{tabular}{@{}c@{}}%
\phantom{0}\\\phantom{0}\\\phantom{0}
\end{tabular}\endgroup%
}}\right.$}%
\begingroup \smaller\smaller\smaller\begin{tabular}{@{}c@{}}%
-1/10\\\phantom{0}\\\phantom{0}
\end{tabular}\endgroup%
\kern3pt%
\begingroup \smaller\smaller\smaller\begin{tabular}{@{}c@{}}%
\phantom{0}\\12/5\\\phantom{0}
\end{tabular}\endgroup%
\kern3pt%
\begingroup \smaller\smaller\smaller\begin{tabular}{@{}c@{}}%
\phantom{0}\\\phantom{0}\\3/2
\end{tabular}\endgroup%
{$\left.\llap{\phantom{%
\begingroup \smaller\smaller\smaller\begin{tabular}{@{}c@{}}%
\phantom{0}\\\phantom{0}\\\phantom{0}
\end{tabular}\endgroup%
}}\!\right]$}%
{$\left[\!\llap{\phantom{%
\begingroup \smaller\smaller\smaller\begin{tabular}{@{}c@{}}%
0\\0\\0
\end{tabular}\endgroup%
}}\right.$}%
\begingroup \smaller\smaller\smaller\begin{tabular}{@{}c@{}}%
2\\-1\\0
\end{tabular}\endgroup%
\kern3pt%
\begingroup \smaller\smaller\smaller\begin{tabular}{@{}c@{}}%
12\\-1\\4
\end{tabular}\endgroup%
\kern3pt%
\begingroup \smaller\smaller\smaller\begin{tabular}{@{}c@{}}%
3\\1\\1
\end{tabular}\endgroup%
{$\left.\llap{\phantom{%
\begingroup \smaller\smaller\smaller\begin{tabular}{@{}c@{}}%
0\\0\\0
\end{tabular}\endgroup%
}}\!\right]$}%
}%
\ifdim\wd\matricesbox>\halfwidth\myboxwidth=\hsize\else\myboxwidth=\halfwidth\fi
\vbox{%
\ifdim\myboxwidth=\hsize
\setbox\onelinebox=\hbox{%
\vbox{\hbox{%
$\Pi_{5,21}$ spans $L_{123.7}$%
}\hbox{%
$22|22\slashtwo\rtimes D_{2}$ (shared)%
}%
}%
\hfill\copy\matricesbox
}%
\ifdim\wd\onelinebox>\myboxwidth
\hbox to \myboxwidth{%
$\Pi_{5,21}$ spans $L_{123.7}$%
\hfil
$22|22\slashtwo\rtimes D_{2}$ (shared)%
}%
\box\matricesbox
\else
\hbox to \myboxwidth{%
\unhbox\onelinebox
}%
\fi
\else
\hbox to \myboxwidth{%
$\Pi_{5,21}$ spans $L_{123.7}$%
\hfil}%
\hbox to \myboxwidth{%
$22|22\slashtwo\rtimes D_{2}$ (shared)%
\hfil}%
\box\matricesbox
\fi
}%
\hfill\discretionary{}{}{}%
\setbox\matricesbox=\hbox{%
{$\left[\!\llap{\phantom{%
\begingroup \smaller\smaller\smaller\begin{tabular}{@{}c@{}}%
\phantom{0}\\\phantom{0}\\\phantom{0}
\end{tabular}\endgroup%
}}\right.$}%
\begingroup \smaller\smaller\smaller\begin{tabular}{@{}c@{}}%
-1/13\\\phantom{0}\\\phantom{0}
\end{tabular}\endgroup%
\kern3pt%
\begingroup \smaller\smaller\smaller\begin{tabular}{@{}c@{}}%
\phantom{0}\\3/13\\\phantom{0}
\end{tabular}\endgroup%
\kern3pt%
\begingroup \smaller\smaller\smaller\begin{tabular}{@{}c@{}}%
\phantom{0}\\\phantom{0}\\3
\end{tabular}\endgroup%
{$\left.\llap{\phantom{%
\begingroup \smaller\smaller\smaller\begin{tabular}{@{}c@{}}%
\phantom{0}\\\phantom{0}\\\phantom{0}
\end{tabular}\endgroup%
}}\!\right]$}%
{$\left[\!\llap{\phantom{%
\begingroup \smaller\smaller\smaller\begin{tabular}{@{}c@{}}%
0\\0\\0
\end{tabular}\endgroup%
}}\right.$}%
\begingroup \smaller\smaller\smaller\begin{tabular}{@{}c@{}}%
3\\4\\0
\end{tabular}\endgroup%
\kern3pt%
\begingroup \smaller\smaller\smaller\begin{tabular}{@{}c@{}}%
8\\2\\-2
\end{tabular}\endgroup%
\kern3pt%
\begingroup \smaller\smaller\smaller\begin{tabular}{@{}c@{}}%
6\\-5\\-1
\end{tabular}\endgroup%
{$\left.\llap{\phantom{%
\begingroup \smaller\smaller\smaller\begin{tabular}{@{}c@{}}%
0\\0\\0
\end{tabular}\endgroup%
}}\!\right]$}%
}%
\ifdim\wd\matricesbox>\halfwidth\myboxwidth=\hsize\else\myboxwidth=\halfwidth\fi
\vbox{%
\ifdim\myboxwidth=\hsize
\setbox\onelinebox=\hbox{%
\vbox{\hbox{%
$\Pi_{5,22}$ spans $L_{123.8}$%
}\hbox{%
$|22\slashtwo22\rtimes D_{2}$ (shared)%
}%
}%
\hfill\copy\matricesbox
}%
\ifdim\wd\onelinebox>\myboxwidth
\hbox to \myboxwidth{%
$\Pi_{5,22}$ spans $L_{123.8}$%
\hfil
$|22\slashtwo22\rtimes D_{2}$ (shared)%
}%
\box\matricesbox
\else
\hbox to \myboxwidth{%
\unhbox\onelinebox
}%
\fi
\else
\hbox to \myboxwidth{%
$\Pi_{5,22}$ spans $L_{123.8}$%
\hfil}%
\hbox to \myboxwidth{%
$|22\slashtwo22\rtimes D_{2}$ (shared)%
\hfil}%
\box\matricesbox
\fi
}%
\hfill\discretionary{}{}{}%
\setbox\matricesbox=\hbox{%
{$\left[\!\llap{\phantom{%
\begingroup \smaller\smaller\smaller\begin{tabular}{@{}c@{}}%
\phantom{0}\\\phantom{0}\\\phantom{0}
\end{tabular}\endgroup%
}}\right.$}%
\begingroup \smaller\smaller\smaller\begin{tabular}{@{}c@{}}%
-1/4\\\phantom{0}\\\phantom{0}
\end{tabular}\endgroup%
\kern3pt%
\begingroup \smaller\smaller\smaller\begin{tabular}{@{}c@{}}%
\phantom{0}\\3/4\\\phantom{0}
\end{tabular}\endgroup%
\kern3pt%
\begingroup \smaller\smaller\smaller\begin{tabular}{@{}c@{}}%
\phantom{0}\\\phantom{0}\\1/2
\end{tabular}\endgroup%
{$\left.\llap{\phantom{%
\begingroup \smaller\smaller\smaller\begin{tabular}{@{}c@{}}%
\phantom{0}\\\phantom{0}\\\phantom{0}
\end{tabular}\endgroup%
}}\!\right]$}%
{$\left[\!\llap{\phantom{%
\begingroup \smaller\smaller\smaller\begin{tabular}{@{}c@{}}%
0\\0\\0
\end{tabular}\endgroup%
}}\right.$}%
\begingroup \smaller\smaller\smaller\begin{tabular}{@{}c@{}}%
2\\2\\0
\end{tabular}\endgroup%
\kern3pt%
\begingroup \smaller\smaller\smaller\begin{tabular}{@{}c@{}}%
3\\1\\-3
\end{tabular}\endgroup%
\kern3pt%
\begingroup \smaller\smaller\smaller\begin{tabular}{@{}c@{}}%
1\\-1\\-1
\end{tabular}\endgroup%
{$\left.\llap{\phantom{%
\begingroup \smaller\smaller\smaller\begin{tabular}{@{}c@{}}%
0\\0\\0
\end{tabular}\endgroup%
}}\!\right]$}%
}%
\ifdim\wd\matricesbox>\halfwidth\myboxwidth=\hsize\else\myboxwidth=\halfwidth\fi
\vbox{%
\ifdim\myboxwidth=\hsize
\setbox\onelinebox=\hbox{%
\vbox{\hbox{%
$\Pi_{5,23}$ spans $L_{123.3}$%
}\hbox{%
$|22\slashtwo22\rtimes D_{2}$ (shared)%
}%
}%
\hfill\copy\matricesbox
}%
\ifdim\wd\onelinebox>\myboxwidth
\hbox to \myboxwidth{%
$\Pi_{5,23}$ spans $L_{123.3}$%
\hfil
$|22\slashtwo22\rtimes D_{2}$ (shared)%
}%
\box\matricesbox
\else
\hbox to \myboxwidth{%
\unhbox\onelinebox
}%
\fi
\else
\hbox to \myboxwidth{%
$\Pi_{5,23}$ spans $L_{123.3}$%
\hfil}%
\hbox to \myboxwidth{%
$|22\slashtwo22\rtimes D_{2}$ (shared)%
\hfil}%
\box\matricesbox
\fi
}%
\hfill\discretionary{}{}{}%
\setbox\matricesbox=\hbox{%
{$\left[\!\llap{\phantom{%
\begingroup \smaller\smaller\smaller\begin{tabular}{@{}c@{}}%
\phantom{0}\\\phantom{0}\\\phantom{0}
\end{tabular}\endgroup%
}}\right.$}%
\begingroup \smaller\smaller\smaller\begin{tabular}{@{}c@{}}%
-1/19\\\phantom{0}\\\phantom{0}
\end{tabular}\endgroup%
\kern3pt%
\begingroup \smaller\smaller\smaller\begin{tabular}{@{}c@{}}%
\phantom{0}\\15/38\\\phantom{0}
\end{tabular}\endgroup%
\kern3pt%
\begingroup \smaller\smaller\smaller\begin{tabular}{@{}c@{}}%
\phantom{0}\\\phantom{0}\\15/2
\end{tabular}\endgroup%
{$\left.\llap{\phantom{%
\begingroup \smaller\smaller\smaller\begin{tabular}{@{}c@{}}%
\phantom{0}\\\phantom{0}\\\phantom{0}
\end{tabular}\endgroup%
}}\!\right]$}%
{$\left[\!\llap{\phantom{%
\begingroup \smaller\smaller\smaller\begin{tabular}{@{}c@{}}%
0\\0\\0
\end{tabular}\endgroup%
}}\right.$}%
\begingroup \smaller\smaller\smaller\begin{tabular}{@{}c@{}}%
15\\7\\-1
\end{tabular}\endgroup%
\kern3pt%
\begingroup \smaller\smaller\smaller\begin{tabular}{@{}c@{}}%
6\\-1\\-1
\end{tabular}\endgroup%
\kern3pt%
\begingroup \smaller\smaller\smaller\begin{tabular}{@{}c@{}}%
5\\-4\\0
\end{tabular}\endgroup%
{$\left.\llap{\phantom{%
\begingroup \smaller\smaller\smaller\begin{tabular}{@{}c@{}}%
0\\0\\0
\end{tabular}\endgroup%
}}\!\right]$}%
}%
\ifdim\wd\matricesbox>\halfwidth\myboxwidth=\hsize\else\myboxwidth=\halfwidth\fi
\vbox{%
\ifdim\myboxwidth=\hsize
\setbox\onelinebox=\hbox{%
\vbox{\hbox{%
$\Pi_{5,24}$ spans $L_{19.6}$%
}\hbox{%
$2\slashtwo22|2\rtimes D_{2}$%
}%
}%
\hfill\copy\matricesbox
}%
\ifdim\wd\onelinebox>\myboxwidth
\hbox to \myboxwidth{%
$\Pi_{5,24}$ spans $L_{19.6}$%
\hfil
$2\slashtwo22|2\rtimes D_{2}$%
}%
\box\matricesbox
\else
\hbox to \myboxwidth{%
\unhbox\onelinebox
}%
\fi
\else
\hbox to \myboxwidth{%
$\Pi_{5,24}$ spans $L_{19.6}$%
\hfil}%
\hbox to \myboxwidth{%
$2\slashtwo22|2\rtimes D_{2}$%
\hfil}%
\box\matricesbox
\fi
}%
\hfill\discretionary{}{}{}%
\setbox\matricesbox=\hbox{%
{$\left[\!\llap{\phantom{%
\begingroup \smaller\smaller\smaller\begin{tabular}{@{}c@{}}%
\phantom{0}\\\phantom{0}\\\phantom{0}
\end{tabular}\endgroup%
}}\right.$}%
\begingroup \smaller\smaller\smaller\begin{tabular}{@{}c@{}}%
-1/5\\\phantom{0}\\\phantom{0}
\end{tabular}\endgroup%
\kern3pt%
\begingroup \smaller\smaller\smaller\begin{tabular}{@{}c@{}}%
\phantom{0}\\1/5\\\phantom{0}
\end{tabular}\endgroup%
\kern3pt%
\begingroup \smaller\smaller\smaller\begin{tabular}{@{}c@{}}%
\phantom{0}\\\phantom{0}\\7
\end{tabular}\endgroup%
{$\left.\llap{\phantom{%
\begingroup \smaller\smaller\smaller\begin{tabular}{@{}c@{}}%
\phantom{0}\\\phantom{0}\\\phantom{0}
\end{tabular}\endgroup%
}}\!\right]$}%
{$\left[\!\llap{\phantom{%
\begingroup \smaller\smaller\smaller\begin{tabular}{@{}c@{}}%
0\\0\\0
\end{tabular}\endgroup%
}}\right.$}%
\begingroup \smaller\smaller\smaller\begin{tabular}{@{}c@{}}%
4\\6\\0
\end{tabular}\endgroup%
\kern3pt%
\begingroup \smaller\smaller\smaller\begin{tabular}{@{}c@{}}%
4\\1\\1
\end{tabular}\endgroup%
\kern3pt%
\begingroup \smaller\smaller\smaller\begin{tabular}{@{}c@{}}%
7\\-7\\1
\end{tabular}\endgroup%
{$\left.\llap{\phantom{%
\begingroup \smaller\smaller\smaller\begin{tabular}{@{}c@{}}%
0\\0\\0
\end{tabular}\endgroup%
}}\!\right]$}%
}%
\ifdim\wd\matricesbox>\halfwidth\myboxwidth=\hsize\else\myboxwidth=\halfwidth\fi
\vbox{%
\ifdim\myboxwidth=\hsize
\setbox\onelinebox=\hbox{%
\vbox{\hbox{%
$\Pi_{5,25}$ spans $L_{15.3}$%
}\hbox{%
$3|32\slashinfty2\rtimes D_{2}$%
}%
}%
\hfill\copy\matricesbox
}%
\ifdim\wd\onelinebox>\myboxwidth
\hbox to \myboxwidth{%
$\Pi_{5,25}$ spans $L_{15.3}$%
\hfil
$3|32\slashinfty2\rtimes D_{2}$%
}%
\box\matricesbox
\else
\hbox to \myboxwidth{%
\unhbox\onelinebox
}%
\fi
\else
\hbox to \myboxwidth{%
$\Pi_{5,25}$ spans $L_{15.3}$%
\hfil}%
\hbox to \myboxwidth{%
$3|32\slashinfty2\rtimes D_{2}$%
\hfil}%
\box\matricesbox
\fi
}%
\hfill\discretionary{}{}{}%
\setbox\matricesbox=\hbox{%
{$\left[\!\llap{\phantom{%
\begingroup \smaller\smaller\smaller\begin{tabular}{@{}c@{}}%
\phantom{0}\\\phantom{0}\\\phantom{0}
\end{tabular}\endgroup%
}}\right.$}%
\begingroup \smaller\smaller\smaller\begin{tabular}{@{}c@{}}%
-5/14\\\phantom{0}\\\phantom{0}
\end{tabular}\endgroup%
\kern3pt%
\begingroup \smaller\smaller\smaller\begin{tabular}{@{}c@{}}%
\phantom{0}\\3/14\\\phantom{0}
\end{tabular}\endgroup%
\kern3pt%
\begingroup \smaller\smaller\smaller\begin{tabular}{@{}c@{}}%
\phantom{0}\\\phantom{0}\\3/2
\end{tabular}\endgroup%
{$\left.\llap{\phantom{%
\begingroup \smaller\smaller\smaller\begin{tabular}{@{}c@{}}%
\phantom{0}\\\phantom{0}\\\phantom{0}
\end{tabular}\endgroup%
}}\!\right]$}%
{$\left[\!\llap{\phantom{%
\begingroup \smaller\smaller\smaller\begin{tabular}{@{}c@{}}%
0\\0\\0
\end{tabular}\endgroup%
}}\right.$}%
\begingroup \smaller\smaller\smaller\begin{tabular}{@{}c@{}}%
2\\-3\\1
\end{tabular}\endgroup%
\kern3pt%
\begingroup \smaller\smaller\smaller\begin{tabular}{@{}c@{}}%
6\\5\\3
\end{tabular}\endgroup%
\kern3pt%
\begingroup \smaller\smaller\smaller\begin{tabular}{@{}c@{}}%
2\\4\\0
\end{tabular}\endgroup%
{$\left.\llap{\phantom{%
\begingroup \smaller\smaller\smaller\begin{tabular}{@{}c@{}}%
0\\0\\0
\end{tabular}\endgroup%
}}\!\right]$}%
}%
\ifdim\wd\matricesbox>\halfwidth\myboxwidth=\hsize\else\myboxwidth=\halfwidth\fi
\vbox{%
\ifdim\myboxwidth=\hsize
\setbox\onelinebox=\hbox{%
\vbox{\hbox{%
$\Pi_{5,26}$ spans $L_{31.4}$%
}\hbox{%
$\slashthree62|26\rtimes D_{2}$%
}%
}%
\hfill\copy\matricesbox
}%
\ifdim\wd\onelinebox>\myboxwidth
\hbox to \myboxwidth{%
$\Pi_{5,26}$ spans $L_{31.4}$%
\hfil
$\slashthree62|26\rtimes D_{2}$%
}%
\box\matricesbox
\else
\hbox to \myboxwidth{%
\unhbox\onelinebox
}%
\fi
\else
\hbox to \myboxwidth{%
$\Pi_{5,26}$ spans $L_{31.4}$%
\hfil}%
\hbox to \myboxwidth{%
$\slashthree62|26\rtimes D_{2}$%
\hfil}%
\box\matricesbox
\fi
}%
\hfill\discretionary{}{}{}%
\setbox\matricesbox=\hbox{%
{$\left[\!\llap{\phantom{%
\begingroup \smaller\smaller\smaller\begin{tabular}{@{}c@{}}%
\phantom{0}\\\phantom{0}\\\phantom{0}
\end{tabular}\endgroup%
}}\right.$}%
\begingroup \smaller\smaller\smaller\begin{tabular}{@{}c@{}}%
-1/9\\\phantom{0}\\\phantom{0}
\end{tabular}\endgroup%
\kern3pt%
\begingroup \smaller\smaller\smaller\begin{tabular}{@{}c@{}}%
\phantom{0}\\5/18\\\phantom{0}
\end{tabular}\endgroup%
\kern3pt%
\begingroup \smaller\smaller\smaller\begin{tabular}{@{}c@{}}%
\phantom{0}\\\phantom{0}\\15/2
\end{tabular}\endgroup%
{$\left.\llap{\phantom{%
\begingroup \smaller\smaller\smaller\begin{tabular}{@{}c@{}}%
\phantom{0}\\\phantom{0}\\\phantom{0}
\end{tabular}\endgroup%
}}\!\right]$}%
{$\left[\!\llap{\phantom{%
\begingroup \smaller\smaller\smaller\begin{tabular}{@{}c@{}}%
0\\0\\0
\end{tabular}\endgroup%
}}\right.$}%
\begingroup \smaller\smaller\smaller\begin{tabular}{@{}c@{}}%
10\\-7\\1
\end{tabular}\endgroup%
\kern3pt%
\begingroup \smaller\smaller\smaller\begin{tabular}{@{}c@{}}%
30\\6\\4
\end{tabular}\endgroup%
\kern3pt%
\begingroup \smaller\smaller\smaller\begin{tabular}{@{}c@{}}%
1\\2\\0
\end{tabular}\endgroup%
{$\left.\llap{\phantom{%
\begingroup \smaller\smaller\smaller\begin{tabular}{@{}c@{}}%
0\\0\\0
\end{tabular}\endgroup%
}}\!\right]$}%
}%
\ifdim\wd\matricesbox>\halfwidth\myboxwidth=\hsize\else\myboxwidth=\halfwidth\fi
\vbox{%
\ifdim\myboxwidth=\hsize
\setbox\onelinebox=\hbox{%
\vbox{\hbox{%
$\Pi_{5,27}$ spans $L_{16.14}$%
}\hbox{%
$\slashthree62|26\rtimes D_{2}$%
}%
}%
\hfill\copy\matricesbox
}%
\ifdim\wd\onelinebox>\myboxwidth
\hbox to \myboxwidth{%
$\Pi_{5,27}$ spans $L_{16.14}$%
\hfil
$\slashthree62|26\rtimes D_{2}$%
}%
\box\matricesbox
\else
\hbox to \myboxwidth{%
\unhbox\onelinebox
}%
\fi
\else
\hbox to \myboxwidth{%
$\Pi_{5,27}$ spans $L_{16.14}$%
\hfil}%
\hbox to \myboxwidth{%
$\slashthree62|26\rtimes D_{2}$%
\hfil}%
\box\matricesbox
\fi
}%
\hfill\discretionary{}{}{}%
\setbox\matricesbox=\hbox{%
{$\left[\!\llap{\phantom{%
\begingroup \smaller\smaller\smaller\begin{tabular}{@{}c@{}}%
\phantom{0}\\\phantom{0}\\\phantom{0}
\end{tabular}\endgroup%
}}\right.$}%
\begingroup \smaller\smaller\smaller\begin{tabular}{@{}c@{}}%
-1/8\\\phantom{0}\\\phantom{0}
\end{tabular}\endgroup%
\kern3pt%
\begingroup \smaller\smaller\smaller\begin{tabular}{@{}c@{}}%
\phantom{0}\\9/8\\\phantom{0}
\end{tabular}\endgroup%
\kern3pt%
\begingroup \smaller\smaller\smaller\begin{tabular}{@{}c@{}}%
\phantom{0}\\\phantom{0}\\18
\end{tabular}\endgroup%
{$\left.\llap{\phantom{%
\begingroup \smaller\smaller\smaller\begin{tabular}{@{}c@{}}%
\phantom{0}\\\phantom{0}\\\phantom{0}
\end{tabular}\endgroup%
}}\!\right]$}%
{$\left[\!\llap{\phantom{%
\begingroup \smaller\smaller\smaller\begin{tabular}{@{}c@{}}%
0\\0\\0
\end{tabular}\endgroup%
}}\right.$}%
\begingroup \smaller\smaller\smaller\begin{tabular}{@{}c@{}}%
18\\6\\-1
\end{tabular}\endgroup%
\kern3pt%
\begingroup \smaller\smaller\smaller\begin{tabular}{@{}c@{}}%
9\\-1\\-1
\end{tabular}\endgroup%
\kern3pt%
\begingroup \smaller\smaller\smaller\begin{tabular}{@{}c@{}}%
1\\-1\\0
\end{tabular}\endgroup%
{$\left.\llap{\phantom{%
\begingroup \smaller\smaller\smaller\begin{tabular}{@{}c@{}}%
0\\0\\0
\end{tabular}\endgroup%
}}\!\right]$}%
}%
\ifdim\wd\matricesbox>\halfwidth\myboxwidth=\hsize\else\myboxwidth=\halfwidth\fi
\vbox{%
\ifdim\myboxwidth=\hsize
\setbox\onelinebox=\hbox{%
\vbox{\hbox{%
$\Pi_{5,28}$ spans $L_{142.18}$%
}\hbox{%
$4\slashinfty42|2\rtimes D_{2}$%
}%
}%
\hfill\copy\matricesbox
}%
\ifdim\wd\onelinebox>\myboxwidth
\hbox to \myboxwidth{%
$\Pi_{5,28}$ spans $L_{142.18}$%
\hfil
$4\slashinfty42|2\rtimes D_{2}$%
}%
\box\matricesbox
\else
\hbox to \myboxwidth{%
\unhbox\onelinebox
}%
\fi
\else
\hbox to \myboxwidth{%
$\Pi_{5,28}$ spans $L_{142.18}$%
\hfil}%
\hbox to \myboxwidth{%
$4\slashinfty42|2\rtimes D_{2}$%
\hfil}%
\box\matricesbox
\fi
}%
\hfill\discretionary{}{}{}%
\setbox\matricesbox=\hbox{%
{$\left[\!\llap{\phantom{%
\begingroup \smaller\smaller\smaller\begin{tabular}{@{}c@{}}%
\phantom{0}\\\phantom{0}\\\phantom{0}
\end{tabular}\endgroup%
}}\right.$}%
\begingroup \smaller\smaller\smaller\begin{tabular}{@{}c@{}}%
-1/5\\\phantom{0}\\\phantom{0}
\end{tabular}\endgroup%
\kern3pt%
\begingroup \smaller\smaller\smaller\begin{tabular}{@{}c@{}}%
\phantom{0}\\1/5\\\phantom{0}
\end{tabular}\endgroup%
\kern3pt%
\begingroup \smaller\smaller\smaller\begin{tabular}{@{}c@{}}%
\phantom{0}\\\phantom{0}\\5
\end{tabular}\endgroup%
{$\left.\llap{\phantom{%
\begingroup \smaller\smaller\smaller\begin{tabular}{@{}c@{}}%
\phantom{0}\\\phantom{0}\\\phantom{0}
\end{tabular}\endgroup%
}}\!\right]$}%
{$\left[\!\llap{\phantom{%
\begingroup \smaller\smaller\smaller\begin{tabular}{@{}c@{}}%
0\\0\\0
\end{tabular}\endgroup%
}}\right.$}%
\begingroup \smaller\smaller\smaller\begin{tabular}{@{}c@{}}%
4\\6\\0
\end{tabular}\endgroup%
\kern3pt%
\begingroup \smaller\smaller\smaller\begin{tabular}{@{}c@{}}%
8\\2\\-2
\end{tabular}\endgroup%
\kern3pt%
\begingroup \smaller\smaller\smaller\begin{tabular}{@{}c@{}}%
5\\-5\\-1
\end{tabular}\endgroup%
{$\left.\llap{\phantom{%
\begingroup \smaller\smaller\smaller\begin{tabular}{@{}c@{}}%
0\\0\\0
\end{tabular}\endgroup%
}}\!\right]$}%
}%
\ifdim\wd\matricesbox>\halfwidth\myboxwidth=\hsize\else\myboxwidth=\halfwidth\fi
\vbox{%
\ifdim\myboxwidth=\hsize
\setbox\onelinebox=\hbox{%
\vbox{\hbox{%
$\Pi_{5,29}$ spans $L_{5.11}$%
}\hbox{%
$4|42\slashinfty2\rtimes D_{2}$%
}%
}%
\hfill\copy\matricesbox
}%
\ifdim\wd\onelinebox>\myboxwidth
\hbox to \myboxwidth{%
$\Pi_{5,29}$ spans $L_{5.11}$%
\hfil
$4|42\slashinfty2\rtimes D_{2}$%
}%
\box\matricesbox
\else
\hbox to \myboxwidth{%
\unhbox\onelinebox
}%
\fi
\else
\hbox to \myboxwidth{%
$\Pi_{5,29}$ spans $L_{5.11}$%
\hfil}%
\hbox to \myboxwidth{%
$4|42\slashinfty2\rtimes D_{2}$%
\hfil}%
\box\matricesbox
\fi
}%
\hfill\discretionary{}{}{}%
\setbox\matricesbox=\hbox{%
{$\left[\!\llap{\phantom{%
\begingroup \smaller\smaller\smaller\begin{tabular}{@{}c@{}}%
\phantom{0}\\\phantom{0}\\\phantom{0}
\end{tabular}\endgroup%
}}\right.$}%
\begingroup \smaller\smaller\smaller\begin{tabular}{@{}c@{}}%
-1\\\phantom{0}\\\phantom{0}
\end{tabular}\endgroup%
\kern3pt%
\begingroup \smaller\smaller\smaller\begin{tabular}{@{}c@{}}%
\phantom{0}\\2\\\phantom{0}
\end{tabular}\endgroup%
\kern3pt%
\begingroup \smaller\smaller\smaller\begin{tabular}{@{}c@{}}%
\phantom{0}\\\phantom{0}\\2
\end{tabular}\endgroup%
{$\left.\llap{\phantom{%
\begingroup \smaller\smaller\smaller\begin{tabular}{@{}c@{}}%
\phantom{0}\\\phantom{0}\\\phantom{0}
\end{tabular}\endgroup%
}}\!\right]$}%
{$\left[\!\llap{\phantom{%
\begingroup \smaller\smaller\smaller\begin{tabular}{@{}c@{}}%
0\\0\\0
\end{tabular}\endgroup%
}}\right.$}%
\begingroup \smaller\smaller\smaller\begin{tabular}{@{}c@{}}%
1\\1\\0
\end{tabular}\endgroup%
\kern3pt%
\begingroup \smaller\smaller\smaller\begin{tabular}{@{}c@{}}%
1\\0\\1
\end{tabular}\endgroup%
\kern3pt%
\begingroup \smaller\smaller\smaller\begin{tabular}{@{}c@{}}%
4\\-3\\1
\end{tabular}\endgroup%
{$\left.\llap{\phantom{%
\begingroup \smaller\smaller\smaller\begin{tabular}{@{}c@{}}%
0\\0\\0
\end{tabular}\endgroup%
}}\!\right]$}%
}%
\ifdim\wd\matricesbox>\halfwidth\myboxwidth=\hsize\else\myboxwidth=\halfwidth\fi
\vbox{%
\ifdim\myboxwidth=\hsize
\setbox\onelinebox=\hbox{%
\vbox{\hbox{%
$\Pi_{5,30}=\hbox{GN}_{36}$ spans $L_{1.9}$%
}\hbox{%
$\infty|\infty\infty\slashtwo\infty\rtimes D_{2}$%
}%
}%
\hfill\copy\matricesbox
}%
\ifdim\wd\onelinebox>\myboxwidth
\hbox to \myboxwidth{%
$\Pi_{5,30}=\hbox{GN}_{36}$ spans $L_{1.9}$%
\hfil
$\infty|\infty\infty\slashtwo\infty\rtimes D_{2}$%
}%
\box\matricesbox
\else
\hbox to \myboxwidth{%
\unhbox\onelinebox
}%
\fi
\else
\hbox to \myboxwidth{%
$\Pi_{5,30}=\hbox{GN}_{36}$ spans $L_{1.9}$%
\hfil}%
\hbox to \myboxwidth{%
$\infty|\infty\infty\slashtwo\infty\rtimes D_{2}$%
\hfil}%
\box\matricesbox
\fi
}%
\hfill\discretionary{}{}{}%
\setbox\matricesbox=\hbox{%
{$\left[\!\llap{\phantom{%
\begingroup \smaller\smaller\smaller\begin{tabular}{@{}c@{}}%
\phantom{0}\\\phantom{0}\\\phantom{0}
\end{tabular}\endgroup%
}}\right.$}%
\begingroup \smaller\smaller\smaller\begin{tabular}{@{}c@{}}%
-9/14\\\phantom{0}\\\phantom{0}
\end{tabular}\endgroup%
\kern3pt%
\begingroup \smaller\smaller\smaller\begin{tabular}{@{}c@{}}%
\phantom{0}\\1/14\\\phantom{0}
\end{tabular}\endgroup%
\kern3pt%
\begingroup \smaller\smaller\smaller\begin{tabular}{@{}c@{}}%
\phantom{0}\\\phantom{0}\\1/2
\end{tabular}\endgroup%
{$\left.\llap{\phantom{%
\begingroup \smaller\smaller\smaller\begin{tabular}{@{}c@{}}%
\phantom{0}\\\phantom{0}\\\phantom{0}
\end{tabular}\endgroup%
}}\!\right]$}%
{$\left[\!\llap{\phantom{%
\begingroup \smaller\smaller\smaller\begin{tabular}{@{}c@{}}%
0\\0\\0
\end{tabular}\endgroup%
}}\right.$}%
\begingroup \smaller\smaller\smaller\begin{tabular}{@{}c@{}}%
2\\-8\\0
\end{tabular}\endgroup%
\kern3pt%
\begingroup \smaller\smaller\smaller\begin{tabular}{@{}c@{}}%
2\\-1\\-3
\end{tabular}\endgroup%
\kern3pt%
\begingroup \smaller\smaller\smaller\begin{tabular}{@{}c@{}}%
2\\6\\-2
\end{tabular}\endgroup%
{$\left.\llap{\phantom{%
\begingroup \smaller\smaller\smaller\begin{tabular}{@{}c@{}}%
0\\0\\0
\end{tabular}\endgroup%
}}\!\right]$}%
}%
\ifdim\wd\matricesbox>\halfwidth\myboxwidth=\hsize\else\myboxwidth=\halfwidth\fi
\vbox{%
\ifdim\myboxwidth=\hsize
\setbox\onelinebox=\hbox{%
\vbox{\hbox{%
$\Pi_{5,31}=A_{1,III}=\hbox{GN}_{42}$ spans $L_{148.4}$%
}\hbox{%
$\infty|\infty2\slashinfty2\rtimes D_{2}$ (shared)%
}%
}%
\hfill\copy\matricesbox
}%
\ifdim\wd\onelinebox>\myboxwidth
\hbox to \myboxwidth{%
$\Pi_{5,31}=A_{1,III}=\hbox{GN}_{42}$ spans $L_{148.4}$%
\hfil
$\infty|\infty2\slashinfty2\rtimes D_{2}$ (shared)%
}%
\box\matricesbox
\else
\hbox to \myboxwidth{%
\unhbox\onelinebox
}%
\fi
\else
\hbox to \myboxwidth{%
$\Pi_{5,31}=A_{1,III}=\hbox{GN}_{42}$ spans $L_{148.4}$%
\hfil}%
\hbox to \myboxwidth{%
$\infty|\infty2\slashinfty2\rtimes D_{2}$ (shared)%
\hfil}%
\box\matricesbox
\fi
}%
\hfill\discretionary{}{}{}%
\setbox\matricesbox=\hbox{%
{$\left[\!\llap{\phantom{%
\begingroup \smaller\smaller\smaller\begin{tabular}{@{}c@{}}%
\phantom{0}\\\phantom{0}\\\phantom{0}
\end{tabular}\endgroup%
}}\right.$}%
\begingroup \smaller\smaller\smaller\begin{tabular}{@{}c@{}}%
-1/6\\\phantom{0}\\\phantom{0}
\end{tabular}\endgroup%
\kern3pt%
\begingroup \smaller\smaller\smaller\begin{tabular}{@{}c@{}}%
\phantom{0}\\1/6\\\phantom{0}
\end{tabular}\endgroup%
\kern3pt%
\begingroup \smaller\smaller\smaller\begin{tabular}{@{}c@{}}%
\phantom{0}\\\phantom{0}\\5/2
\end{tabular}\endgroup%
{$\left.\llap{\phantom{%
\begingroup \smaller\smaller\smaller\begin{tabular}{@{}c@{}}%
\phantom{0}\\\phantom{0}\\\phantom{0}
\end{tabular}\endgroup%
}}\!\right]$}%
{$\left[\!\llap{\phantom{%
\begingroup \smaller\smaller\smaller\begin{tabular}{@{}c@{}}%
0\\0\\0
\end{tabular}\endgroup%
}}\right.$}%
\begingroup \smaller\smaller\smaller\begin{tabular}{@{}c@{}}%
10\\10\\2
\end{tabular}\endgroup%
\kern3pt%
\begingroup \smaller\smaller\smaller\begin{tabular}{@{}c@{}}%
10\\-5\\3
\end{tabular}\endgroup%
\kern3pt%
\begingroup \smaller\smaller\smaller\begin{tabular}{@{}c@{}}%
2\\-4\\0
\end{tabular}\endgroup%
{$\left.\llap{\phantom{%
\begingroup \smaller\smaller\smaller\begin{tabular}{@{}c@{}}%
0\\0\\0
\end{tabular}\endgroup%
}}\!\right]$}%
}%
\ifdim\wd\matricesbox>\halfwidth\myboxwidth=\hsize\else\myboxwidth=\halfwidth\fi
\vbox{%
\ifdim\myboxwidth=\hsize
\setbox\onelinebox=\hbox{%
\vbox{\hbox{%
$\Pi_{5,32}$ spans $L_{10.4}$%
}\hbox{%
$\infty\slashinfty\infty2|2\rtimes D_{2}$%
}%
}%
\hfill\copy\matricesbox
}%
\ifdim\wd\onelinebox>\myboxwidth
\hbox to \myboxwidth{%
$\Pi_{5,32}$ spans $L_{10.4}$%
\hfil
$\infty\slashinfty\infty2|2\rtimes D_{2}$%
}%
\box\matricesbox
\else
\hbox to \myboxwidth{%
\unhbox\onelinebox
}%
\fi
\else
\hbox to \myboxwidth{%
$\Pi_{5,32}$ spans $L_{10.4}$%
\hfil}%
\hbox to \myboxwidth{%
$\infty\slashinfty\infty2|2\rtimes D_{2}$%
\hfil}%
\box\matricesbox
\fi
}%
\hfill\discretionary{}{}{}%
\setbox\matricesbox=\hbox{%
{$\left[\!\llap{\phantom{%
\begingroup \smaller\smaller\smaller\begin{tabular}{@{}c@{}}%
\phantom{0}\\\phantom{0}\\\phantom{0}
\end{tabular}\endgroup%
}}\right.$}%
\begingroup \smaller\smaller\smaller\begin{tabular}{@{}c@{}}%
-1/3\\\phantom{0}\\\phantom{0}
\end{tabular}\endgroup%
\kern3pt%
\begingroup \smaller\smaller\smaller\begin{tabular}{@{}c@{}}%
\phantom{0}\\1/3\\\phantom{0}
\end{tabular}\endgroup%
\kern3pt%
\begingroup \smaller\smaller\smaller\begin{tabular}{@{}c@{}}%
\phantom{0}\\\phantom{0}\\3
\end{tabular}\endgroup%
{$\left.\llap{\phantom{%
\begingroup \smaller\smaller\smaller\begin{tabular}{@{}c@{}}%
\phantom{0}\\\phantom{0}\\\phantom{0}
\end{tabular}\endgroup%
}}\!\right]$}%
{$\left[\!\llap{\phantom{%
\begingroup \smaller\smaller\smaller\begin{tabular}{@{}c@{}}%
0\\0\\0
\end{tabular}\endgroup%
}}\right.$}%
\begingroup \smaller\smaller\smaller\begin{tabular}{@{}c@{}}%
3\\-3\\1
\end{tabular}\endgroup%
\kern3pt%
\begingroup \smaller\smaller\smaller\begin{tabular}{@{}c@{}}%
12\\6\\4
\end{tabular}\endgroup%
\kern3pt%
\begingroup \smaller\smaller\smaller\begin{tabular}{@{}c@{}}%
1\\2\\0
\end{tabular}\endgroup%
{$\left.\llap{\phantom{%
\begingroup \smaller\smaller\smaller\begin{tabular}{@{}c@{}}%
0\\0\\0
\end{tabular}\endgroup%
}}\!\right]$}%
}%
\ifdim\wd\matricesbox>\halfwidth\myboxwidth=\hsize\else\myboxwidth=\halfwidth\fi
\vbox{%
\ifdim\myboxwidth=\hsize
\setbox\onelinebox=\hbox{%
\vbox{\hbox{%
$\Pi_{5,33}$ spans $L_{4.27}$%
}\hbox{%
$\slashinfty\infty2|2\infty\rtimes D_{2}$%
}%
}%
\hfill\copy\matricesbox
}%
\ifdim\wd\onelinebox>\myboxwidth
\hbox to \myboxwidth{%
$\Pi_{5,33}$ spans $L_{4.27}$%
\hfil
$\slashinfty\infty2|2\infty\rtimes D_{2}$%
}%
\box\matricesbox
\else
\hbox to \myboxwidth{%
\unhbox\onelinebox
}%
\fi
\else
\hbox to \myboxwidth{%
$\Pi_{5,33}$ spans $L_{4.27}$%
\hfil}%
\hbox to \myboxwidth{%
$\slashinfty\infty2|2\infty\rtimes D_{2}$%
\hfil}%
\box\matricesbox
\fi
}%
\hfill\discretionary{}{}{}%
\setbox\matricesbox=\hbox{%
{$\left[\!\llap{\phantom{%
\begingroup \smaller\smaller\smaller\begin{tabular}{@{}c@{}}%
\phantom{0}\\\phantom{0}\\\phantom{0}
\end{tabular}\endgroup%
}}\right.$}%
\begingroup \smaller\smaller\smaller\begin{tabular}{@{}c@{}}%
-1/3\\\phantom{0}\\\phantom{0}
\end{tabular}\endgroup%
\kern3pt%
\begingroup \smaller\smaller\smaller\begin{tabular}{@{}c@{}}%
\phantom{0}\\1/3\\\phantom{0}
\end{tabular}\endgroup%
\kern3pt%
\begingroup \smaller\smaller\smaller\begin{tabular}{@{}c@{}}%
\phantom{0}\\\phantom{0}\\9
\end{tabular}\endgroup%
{$\left.\llap{\phantom{%
\begingroup \smaller\smaller\smaller\begin{tabular}{@{}c@{}}%
\phantom{0}\\\phantom{0}\\\phantom{0}
\end{tabular}\endgroup%
}}\!\right]$}%
{$\left[\!\llap{\phantom{%
\begingroup \smaller\smaller\smaller\begin{tabular}{@{}c@{}}%
0\\0\\0
\end{tabular}\endgroup%
}}\right.$}%
\begingroup \smaller\smaller\smaller\begin{tabular}{@{}c@{}}%
1\\-2\\0
\end{tabular}\endgroup%
\kern3pt%
\begingroup \smaller\smaller\smaller\begin{tabular}{@{}c@{}}%
4\\1\\-1
\end{tabular}\endgroup%
\kern3pt%
\begingroup \smaller\smaller\smaller\begin{tabular}{@{}c@{}}%
9\\9\\-1
\end{tabular}\endgroup%
{$\left.\llap{\phantom{%
\begingroup \smaller\smaller\smaller\begin{tabular}{@{}c@{}}%
0\\0\\0
\end{tabular}\endgroup%
}}\!\right]$}%
}%
\ifdim\wd\matricesbox>\halfwidth\myboxwidth=\hsize\else\myboxwidth=\halfwidth\fi
\vbox{%
\ifdim\myboxwidth=\hsize
\setbox\onelinebox=\hbox{%
\vbox{\hbox{%
$\Pi_{5,34}=\xi_{0,9}^{(2)}=\hbox{GN}_{26}$ spans $L_{148.9}$%
}\hbox{%
$\infty|\infty2\slashinfty2\rtimes D_{2}$ (shared)%
}%
}%
\hfill\copy\matricesbox
}%
\ifdim\wd\onelinebox>\myboxwidth
\hbox to \myboxwidth{%
$\Pi_{5,34}=\xi_{0,9}^{(2)}=\hbox{GN}_{26}$ spans $L_{148.9}$%
\hfil
$\infty|\infty2\slashinfty2\rtimes D_{2}$ (shared)%
}%
\box\matricesbox
\else
\hbox to \myboxwidth{%
\unhbox\onelinebox
}%
\fi
\else
\hbox to \myboxwidth{%
$\Pi_{5,34}=\xi_{0,9}^{(2)}=\hbox{GN}_{26}$ spans $L_{148.9}$%
\hfil}%
\hbox to \myboxwidth{%
$\infty|\infty2\slashinfty2\rtimes D_{2}$ (shared)%
\hfil}%
\box\matricesbox
\fi
}%
\hfill\discretionary{}{}{}%
\setbox\matricesbox=\hbox{%
{$\left[\!\llap{\phantom{%
\begingroup \smaller\smaller\smaller\begin{tabular}{@{}c@{}}%
\phantom{0}\\\phantom{0}\\\phantom{0}
\end{tabular}\endgroup%
}}\right.$}%
\begingroup \smaller\smaller\smaller\begin{tabular}{@{}c@{}}%
-1/20\\\phantom{0}\\\phantom{0}
\end{tabular}\endgroup%
\kern3pt%
\begingroup \smaller\smaller\smaller\begin{tabular}{@{}c@{}}%
\phantom{0}\\11/5\\\phantom{0}
\end{tabular}\endgroup%
\kern3pt%
\begingroup \smaller\smaller\smaller\begin{tabular}{@{}c@{}}%
\phantom{0}\\\phantom{0}\\11
\end{tabular}\endgroup%
{$\left.\llap{\phantom{%
\begingroup \smaller\smaller\smaller\begin{tabular}{@{}c@{}}%
\phantom{0}\\\phantom{0}\\\phantom{0}
\end{tabular}\endgroup%
}}\!\right]$}%
{$\left[\!\llap{\phantom{%
\begingroup \smaller\smaller\smaller\begin{tabular}{@{}c@{}}%
0\\0\\0
\end{tabular}\endgroup%
}}\right.$}%
\begingroup \smaller\smaller\smaller\begin{tabular}{@{}c@{}}%
22\\-4\\1
\end{tabular}\endgroup%
\kern3pt%
\begingroup \smaller\smaller\smaller\begin{tabular}{@{}c@{}}%
22\\1\\2
\end{tabular}\endgroup%
\kern3pt%
\begingroup \smaller\smaller\smaller\begin{tabular}{@{}c@{}}%
2\\1\\0
\end{tabular}\endgroup%
{$\left.\llap{\phantom{%
\begingroup \smaller\smaller\smaller\begin{tabular}{@{}c@{}}%
0\\0\\0
\end{tabular}\endgroup%
}}\!\right]$}%
}%
\ifdim\wd\matricesbox>\halfwidth\myboxwidth=\hsize\else\myboxwidth=\halfwidth\fi
\vbox{%
\ifdim\myboxwidth=\hsize
\setbox\onelinebox=\hbox{%
\vbox{\hbox{%
$\Pi_{5,35}$ spans $L_{11.5}$%
}\hbox{%
$3\slashtwo32|2\rtimes D_{2}$%
}%
}%
\hfill\copy\matricesbox
}%
\ifdim\wd\onelinebox>\myboxwidth
\hbox to \myboxwidth{%
$\Pi_{5,35}$ spans $L_{11.5}$%
\hfil
$3\slashtwo32|2\rtimes D_{2}$%
}%
\box\matricesbox
\else
\hbox to \myboxwidth{%
\unhbox\onelinebox
}%
\fi
\else
\hbox to \myboxwidth{%
$\Pi_{5,35}$ spans $L_{11.5}$%
\hfil}%
\hbox to \myboxwidth{%
$3\slashtwo32|2\rtimes D_{2}$%
\hfil}%
\box\matricesbox
\fi
}%
\hfill\discretionary{}{}{}%
\setbox\matricesbox=\hbox{%
{$\left[\!\llap{\phantom{%
\begingroup \smaller\smaller\smaller\begin{tabular}{@{}c@{}}%
\phantom{0}\\\phantom{0}\\\phantom{0}
\end{tabular}\endgroup%
}}\right.$}%
\begingroup \smaller\smaller\smaller\begin{tabular}{@{}c@{}}%
-5/19\\\phantom{0}\\\phantom{0}
\end{tabular}\endgroup%
\kern3pt%
\begingroup \smaller\smaller\smaller\begin{tabular}{@{}c@{}}%
\phantom{0}\\3/38\\\phantom{0}
\end{tabular}\endgroup%
\kern3pt%
\begingroup \smaller\smaller\smaller\begin{tabular}{@{}c@{}}%
\phantom{0}\\\phantom{0}\\3/2
\end{tabular}\endgroup%
{$\left.\llap{\phantom{%
\begingroup \smaller\smaller\smaller\begin{tabular}{@{}c@{}}%
\phantom{0}\\\phantom{0}\\\phantom{0}
\end{tabular}\endgroup%
}}\!\right]$}%
{$\left[\!\llap{\phantom{%
\begingroup \smaller\smaller\smaller\begin{tabular}{@{}c@{}}%
0\\0\\0
\end{tabular}\endgroup%
}}\right.$}%
\begingroup \smaller\smaller\smaller\begin{tabular}{@{}c@{}}%
1\\-4\\0
\end{tabular}\endgroup%
\kern3pt%
\begingroup \smaller\smaller\smaller\begin{tabular}{@{}c@{}}%
6\\-5\\-3
\end{tabular}\endgroup%
\kern3pt%
\begingroup \smaller\smaller\smaller\begin{tabular}{@{}c@{}}%
3\\7\\-1
\end{tabular}\endgroup%
{$\left.\llap{\phantom{%
\begingroup \smaller\smaller\smaller\begin{tabular}{@{}c@{}}%
0\\0\\0
\end{tabular}\endgroup%
}}\!\right]$}%
}%
\ifdim\wd\matricesbox>\halfwidth\myboxwidth=\hsize\else\myboxwidth=\halfwidth\fi
\vbox{%
\ifdim\myboxwidth=\hsize
\setbox\onelinebox=\hbox{%
\vbox{\hbox{%
$\Pi_{5,36}$ spans $L_{19.3}$%
}\hbox{%
$42|24\slashtwo\rtimes D_{2}$%
}%
}%
\hfill\copy\matricesbox
}%
\ifdim\wd\onelinebox>\myboxwidth
\hbox to \myboxwidth{%
$\Pi_{5,36}$ spans $L_{19.3}$%
\hfil
$42|24\slashtwo\rtimes D_{2}$%
}%
\box\matricesbox
\else
\hbox to \myboxwidth{%
\unhbox\onelinebox
}%
\fi
\else
\hbox to \myboxwidth{%
$\Pi_{5,36}$ spans $L_{19.3}$%
\hfil}%
\hbox to \myboxwidth{%
$42|24\slashtwo\rtimes D_{2}$%
\hfil}%
\box\matricesbox
\fi
}%
\hfill\discretionary{}{}{}%
\setbox\matricesbox=\hbox{%
{$\left[\!\llap{\phantom{%
\begingroup \smaller\smaller\smaller\begin{tabular}{@{}c@{}}%
\phantom{0}\\\phantom{0}\\\phantom{0}
\end{tabular}\endgroup%
}}\right.$}%
\begingroup \smaller\smaller\smaller\begin{tabular}{@{}c@{}}%
-1/2\\\phantom{0}\\\phantom{0}
\end{tabular}\endgroup%
\kern3pt%
\begingroup \smaller\smaller\smaller\begin{tabular}{@{}c@{}}%
\phantom{0}\\1\\\phantom{0}
\end{tabular}\endgroup%
\kern3pt%
\begingroup \smaller\smaller\smaller\begin{tabular}{@{}c@{}}%
\phantom{0}\\\phantom{0}\\1/2
\end{tabular}\endgroup%
{$\left.\llap{\phantom{%
\begingroup \smaller\smaller\smaller\begin{tabular}{@{}c@{}}%
\phantom{0}\\\phantom{0}\\\phantom{0}
\end{tabular}\endgroup%
}}\!\right]$}%
{$\left[\!\llap{\phantom{%
\begingroup \smaller\smaller\smaller\begin{tabular}{@{}c@{}}%
0\\0\\0
\end{tabular}\endgroup%
}}\right.$}%
\begingroup \smaller\smaller\smaller\begin{tabular}{@{}c@{}}%
2\\2\\0
\end{tabular}\endgroup%
\kern3pt%
\begingroup \smaller\smaller\smaller\begin{tabular}{@{}c@{}}%
4\\2\\-4
\end{tabular}\endgroup%
\kern3pt%
\begingroup \smaller\smaller\smaller\begin{tabular}{@{}c@{}}%
1\\-1\\-1
\end{tabular}\endgroup%
{$\left.\llap{\phantom{%
\begingroup \smaller\smaller\smaller\begin{tabular}{@{}c@{}}%
0\\0\\0
\end{tabular}\endgroup%
}}\!\right]$}%
}%
\ifdim\wd\matricesbox>\halfwidth\myboxwidth=\hsize\else\myboxwidth=\halfwidth\fi
\vbox{%
\ifdim\myboxwidth=\hsize
\setbox\onelinebox=\hbox{%
\vbox{\hbox{%
$\Pi_{5,37}$ spans $L_{1.6}$%
}\hbox{%
$\infty2|2\infty\slashtwo\rtimes D_{2}$ (shared)%
}%
}%
\hfill\copy\matricesbox
}%
\ifdim\wd\onelinebox>\myboxwidth
\hbox to \myboxwidth{%
$\Pi_{5,37}$ spans $L_{1.6}$%
\hfil
$\infty2|2\infty\slashtwo\rtimes D_{2}$ (shared)%
}%
\box\matricesbox
\else
\hbox to \myboxwidth{%
\unhbox\onelinebox
}%
\fi
\else
\hbox to \myboxwidth{%
$\Pi_{5,37}$ spans $L_{1.6}$%
\hfil}%
\hbox to \myboxwidth{%
$\infty2|2\infty\slashtwo\rtimes D_{2}$ (shared)%
\hfil}%
\box\matricesbox
\fi
}%
\hfill\discretionary{}{}{}%
\setbox\matricesbox=\hbox{%
{$\left[\!\llap{\phantom{%
\begingroup \smaller\smaller\smaller\begin{tabular}{@{}c@{}}%
\phantom{0}\\\phantom{0}\\\phantom{0}
\end{tabular}\endgroup%
}}\right.$}%
\begingroup \smaller\smaller\smaller\begin{tabular}{@{}c@{}}%
-4/5\\\phantom{0}\\\phantom{0}
\end{tabular}\endgroup%
\kern3pt%
\begingroup \smaller\smaller\smaller\begin{tabular}{@{}c@{}}%
\phantom{0}\\1/5\\\phantom{0}
\end{tabular}\endgroup%
\kern3pt%
\begingroup \smaller\smaller\smaller\begin{tabular}{@{}c@{}}%
\phantom{0}\\\phantom{0}\\1
\end{tabular}\endgroup%
{$\left.\llap{\phantom{%
\begingroup \smaller\smaller\smaller\begin{tabular}{@{}c@{}}%
\phantom{0}\\\phantom{0}\\\phantom{0}
\end{tabular}\endgroup%
}}\!\right]$}%
{$\left[\!\llap{\phantom{%
\begingroup \smaller\smaller\smaller\begin{tabular}{@{}c@{}}%
0\\0\\0
\end{tabular}\endgroup%
}}\right.$}%
\begingroup \smaller\smaller\smaller\begin{tabular}{@{}c@{}}%
1\\-3\\0
\end{tabular}\endgroup%
\kern3pt%
\begingroup \smaller\smaller\smaller\begin{tabular}{@{}c@{}}%
4\\-2\\-4
\end{tabular}\endgroup%
\kern3pt%
\begingroup \smaller\smaller\smaller\begin{tabular}{@{}c@{}}%
1\\2\\-1
\end{tabular}\endgroup%
{$\left.\llap{\phantom{%
\begingroup \smaller\smaller\smaller\begin{tabular}{@{}c@{}}%
0\\0\\0
\end{tabular}\endgroup%
}}\!\right]$}%
}%
\ifdim\wd\matricesbox>\halfwidth\myboxwidth=\hsize\else\myboxwidth=\halfwidth\fi
\vbox{%
\ifdim\myboxwidth=\hsize
\setbox\onelinebox=\hbox{%
\vbox{\hbox{%
$\Pi_{5,38}=\hbox{GN}_{34}$ spans $L_{146.3}$%
}\hbox{%
$|\infty2\slashinfty2\infty\rtimes D_{2}$%
}%
}%
\hfill\copy\matricesbox
}%
\ifdim\wd\onelinebox>\myboxwidth
\hbox to \myboxwidth{%
$\Pi_{5,38}=\hbox{GN}_{34}$ spans $L_{146.3}$%
\hfil
$|\infty2\slashinfty2\infty\rtimes D_{2}$%
}%
\box\matricesbox
\else
\hbox to \myboxwidth{%
\unhbox\onelinebox
}%
\fi
\else
\hbox to \myboxwidth{%
$\Pi_{5,38}=\hbox{GN}_{34}$ spans $L_{146.3}$%
\hfil}%
\hbox to \myboxwidth{%
$|\infty2\slashinfty2\infty\rtimes D_{2}$%
\hfil}%
\box\matricesbox
\fi
}%
\hfill\discretionary{}{}{}%
\setbox\matricesbox=\hbox{%
{$\left[\!\llap{\phantom{%
\begingroup \smaller\smaller\smaller\begin{tabular}{@{}c@{}}%
\phantom{0}\\\phantom{0}\\\phantom{0}
\end{tabular}\endgroup%
}}\right.$}%
\begingroup \smaller\smaller\smaller\begin{tabular}{@{}c@{}}%
-1/8\\\phantom{0}\\\phantom{0}
\end{tabular}\endgroup%
\kern3pt%
\begingroup \smaller\smaller\smaller\begin{tabular}{@{}c@{}}%
\phantom{0}\\5/8\\\phantom{0}
\end{tabular}\endgroup%
\kern3pt%
\begingroup \smaller\smaller\smaller\begin{tabular}{@{}c@{}}%
\phantom{0}\\\phantom{0}\\1/2
\end{tabular}\endgroup%
{$\left.\llap{\phantom{%
\begingroup \smaller\smaller\smaller\begin{tabular}{@{}c@{}}%
\phantom{0}\\\phantom{0}\\\phantom{0}
\end{tabular}\endgroup%
}}\!\right]$}%
{$\left[\!\llap{\phantom{%
\begingroup \smaller\smaller\smaller\begin{tabular}{@{}c@{}}%
0\\0\\0
\end{tabular}\endgroup%
}}\right.$}%
\begingroup \smaller\smaller\smaller\begin{tabular}{@{}c@{}}%
10\\6\\0
\end{tabular}\endgroup%
\kern3pt%
\begingroup \smaller\smaller\smaller\begin{tabular}{@{}c@{}}%
40\\8\\-20
\end{tabular}\endgroup%
\kern3pt%
\begingroup \smaller\smaller\smaller\begin{tabular}{@{}c@{}}%
1\\-1\\-1
\end{tabular}\endgroup%
{$\left.\llap{\phantom{%
\begingroup \smaller\smaller\smaller\begin{tabular}{@{}c@{}}%
0\\0\\0
\end{tabular}\endgroup%
}}\!\right]$}%
}%
\ifdim\wd\matricesbox>\halfwidth\myboxwidth=\hsize\else\myboxwidth=\halfwidth\fi
\vbox{%
\ifdim\myboxwidth=\hsize
\setbox\onelinebox=\hbox{%
\vbox{\hbox{%
$\Pi_{5,39}$ spans $L_{5.2}$%
}\hbox{%
$|\infty2\slashtwo2\infty\rtimes D_{2}$%
}%
}%
\hfill\copy\matricesbox
}%
\ifdim\wd\onelinebox>\myboxwidth
\hbox to \myboxwidth{%
$\Pi_{5,39}$ spans $L_{5.2}$%
\hfil
$|\infty2\slashtwo2\infty\rtimes D_{2}$%
}%
\box\matricesbox
\else
\hbox to \myboxwidth{%
\unhbox\onelinebox
}%
\fi
\else
\hbox to \myboxwidth{%
$\Pi_{5,39}$ spans $L_{5.2}$%
\hfil}%
\hbox to \myboxwidth{%
$|\infty2\slashtwo2\infty\rtimes D_{2}$%
\hfil}%
\box\matricesbox
\fi
}%
\hfill\discretionary{}{}{}%
\setbox\matricesbox=\hbox{%
{$\left[\!\llap{\phantom{%
\begingroup \smaller\smaller\smaller\begin{tabular}{@{}c@{}}%
\phantom{0}\\\phantom{0}\\\phantom{0}
\end{tabular}\endgroup%
}}\right.$}%
\begingroup \smaller\smaller\smaller\begin{tabular}{@{}c@{}}%
-1/7\\\phantom{0}\\\phantom{0}
\end{tabular}\endgroup%
\kern3pt%
\begingroup \smaller\smaller\smaller\begin{tabular}{@{}c@{}}%
\phantom{0}\\9/14\\\phantom{0}
\end{tabular}\endgroup%
\kern3pt%
\begingroup \smaller\smaller\smaller\begin{tabular}{@{}c@{}}%
\phantom{0}\\\phantom{0}\\9/2
\end{tabular}\endgroup%
{$\left.\llap{\phantom{%
\begingroup \smaller\smaller\smaller\begin{tabular}{@{}c@{}}%
\phantom{0}\\\phantom{0}\\\phantom{0}
\end{tabular}\endgroup%
}}\!\right]$}%
{$\left[\!\llap{\phantom{%
\begingroup \smaller\smaller\smaller\begin{tabular}{@{}c@{}}%
0\\0\\0
\end{tabular}\endgroup%
}}\right.$}%
\begingroup \smaller\smaller\smaller\begin{tabular}{@{}c@{}}%
2\\-2\\0
\end{tabular}\endgroup%
\kern3pt%
\begingroup \smaller\smaller\smaller\begin{tabular}{@{}c@{}}%
9\\-2\\2
\end{tabular}\endgroup%
\kern3pt%
\begingroup \smaller\smaller\smaller\begin{tabular}{@{}c@{}}%
9\\5\\1
\end{tabular}\endgroup%
{$\left.\llap{\phantom{%
\begingroup \smaller\smaller\smaller\begin{tabular}{@{}c@{}}%
0\\0\\0
\end{tabular}\endgroup%
}}\!\right]$}%
}%
\ifdim\wd\matricesbox>\halfwidth\myboxwidth=\hsize\else\myboxwidth=\halfwidth\fi
\vbox{%
\ifdim\myboxwidth=\hsize
\setbox\onelinebox=\hbox{%
\vbox{\hbox{%
$\Pi_{5,40}$ spans $L_{142.8}$%
}\hbox{%
$\infty2|2\infty\slashtwo\rtimes D_{2}$ (shared)%
}%
}%
\hfill\copy\matricesbox
}%
\ifdim\wd\onelinebox>\myboxwidth
\hbox to \myboxwidth{%
$\Pi_{5,40}$ spans $L_{142.8}$%
\hfil
$\infty2|2\infty\slashtwo\rtimes D_{2}$ (shared)%
}%
\box\matricesbox
\else
\hbox to \myboxwidth{%
\unhbox\onelinebox
}%
\fi
\else
\hbox to \myboxwidth{%
$\Pi_{5,40}$ spans $L_{142.8}$%
\hfil}%
\hbox to \myboxwidth{%
$\infty2|2\infty\slashtwo\rtimes D_{2}$ (shared)%
\hfil}%
\box\matricesbox
\fi
}%
\hfill\discretionary{}{}{}%
\setbox\matricesbox=\hbox{%
{$\left[\!\llap{\phantom{%
\begingroup \smaller\smaller\smaller\begin{tabular}{@{}c@{}}%
\phantom{0}\\\phantom{0}\\\phantom{0}
\end{tabular}\endgroup%
}}\right.$}%
\begingroup \smaller\smaller\smaller\begin{tabular}{@{}c@{}}%
-1/9\\\phantom{0}\\\phantom{0}
\end{tabular}\endgroup%
\kern3pt%
\begingroup \smaller\smaller\smaller\begin{tabular}{@{}c@{}}%
\phantom{0}\\10/9\\-5/9
\end{tabular}\endgroup%
\kern3pt%
\begingroup \smaller\smaller\smaller\begin{tabular}{@{}c@{}}%
\phantom{0}\\-5/9\\70/9
\end{tabular}\endgroup%
{$\left.\llap{\phantom{%
\begingroup \smaller\smaller\smaller\begin{tabular}{@{}c@{}}%
\phantom{0}\\\phantom{0}\\\phantom{0}
\end{tabular}\endgroup%
}}\!\right]$}%
{$\left[\!\llap{\phantom{%
\begingroup \smaller\smaller\smaller\begin{tabular}{@{}c@{}}%
0\\0\\0
\end{tabular}\endgroup%
}}\right.$}%
\begingroup \smaller\smaller\smaller\begin{tabular}{@{}c@{}}%
1\\1\\0
\end{tabular}\endgroup%
\kern3pt%
\begingroup \smaller\smaller\smaller\begin{tabular}{@{}c@{}}%
5\\1\\1
\end{tabular}\endgroup%
\kern3pt%
\begingroup \smaller\smaller\smaller\begin{tabular}{@{}c@{}}%
10\\-3\\1
\end{tabular}\endgroup%
\kern3pt%
\begingroup \smaller\smaller\smaller\begin{tabular}{@{}c@{}}%
10\\-4\\-1
\end{tabular}\endgroup%
\kern3pt%
\begingroup \smaller\smaller\smaller\begin{tabular}{@{}c@{}}%
30\\1\\-4
\end{tabular}\endgroup%
{$\left.\llap{\phantom{%
\begingroup \smaller\smaller\smaller\begin{tabular}{@{}c@{}}%
0\\0\\0
\end{tabular}\endgroup%
}}\!\right]$}%
}%
\ifdim\wd\matricesbox>\halfwidth\myboxwidth=\hsize\else\myboxwidth=\halfwidth\fi
\vbox{%
\ifdim\myboxwidth=\hsize
\setbox\onelinebox=\hbox{%
\vbox{\hbox{%
$\Pi_{5,41}$ spans $L_{16.6}$%
}\hbox{%
$22362$%
}%
}%
\hfill\copy\matricesbox
}%
\ifdim\wd\onelinebox>\myboxwidth
\hbox to \myboxwidth{%
$\Pi_{5,41}$ spans $L_{16.6}$%
\hfil
$22362$%
}%
\box\matricesbox
\else
\hbox to \myboxwidth{%
\unhbox\onelinebox
}%
\fi
\else
\hbox to \myboxwidth{%
$\Pi_{5,41}$ spans $L_{16.6}$%
\hfil}%
\hbox to \myboxwidth{%
$22362$%
\hfil}%
\box\matricesbox
\fi
}%
\hfill\discretionary{}{}{}%
\setbox\matricesbox=\hbox{%
{$\left[\!\llap{\phantom{%
\begingroup \smaller\smaller\smaller\begin{tabular}{@{}c@{}}%
\phantom{0}\\\phantom{0}\\\phantom{0}
\end{tabular}\endgroup%
}}\right.$}%
\begingroup \smaller\smaller\smaller\begin{tabular}{@{}c@{}}%
-1/9\\\phantom{0}\\\phantom{0}
\end{tabular}\endgroup%
\kern3pt%
\begingroup \smaller\smaller\smaller\begin{tabular}{@{}c@{}}%
\phantom{0}\\10/9\\-2/9
\end{tabular}\endgroup%
\kern3pt%
\begingroup \smaller\smaller\smaller\begin{tabular}{@{}c@{}}%
\phantom{0}\\-2/9\\22/9
\end{tabular}\endgroup%
{$\left.\llap{\phantom{%
\begingroup \smaller\smaller\smaller\begin{tabular}{@{}c@{}}%
\phantom{0}\\\phantom{0}\\\phantom{0}
\end{tabular}\endgroup%
}}\!\right]$}%
{$\left[\!\llap{\phantom{%
\begingroup \smaller\smaller\smaller\begin{tabular}{@{}c@{}}%
0\\0\\0
\end{tabular}\endgroup%
}}\right.$}%
\begingroup \smaller\smaller\smaller\begin{tabular}{@{}c@{}}%
2\\0\\-1
\end{tabular}\endgroup%
\kern3pt%
\begingroup \smaller\smaller\smaller\begin{tabular}{@{}c@{}}%
16\\-6\\-2
\end{tabular}\endgroup%
\kern3pt%
\begingroup \smaller\smaller\smaller\begin{tabular}{@{}c@{}}%
24\\-8\\2
\end{tabular}\endgroup%
\kern3pt%
\begingroup \smaller\smaller\smaller\begin{tabular}{@{}c@{}}%
6\\1\\2
\end{tabular}\endgroup%
\kern3pt%
\begingroup \smaller\smaller\smaller\begin{tabular}{@{}c@{}}%
1\\1\\0
\end{tabular}\endgroup%
{$\left.\llap{\phantom{%
\begingroup \smaller\smaller\smaller\begin{tabular}{@{}c@{}}%
0\\0\\0
\end{tabular}\endgroup%
}}\!\right]$}%
}%
\ifdim\wd\matricesbox>\halfwidth\myboxwidth=\hsize\else\myboxwidth=\halfwidth\fi
\vbox{%
\ifdim\myboxwidth=\hsize
\setbox\onelinebox=\hbox{%
\vbox{\hbox{%
$\Pi_{5,42}$ spans $L_{150.6}$%
}\hbox{%
$22\infty22$ (shared)%
}%
}%
\hfill\copy\matricesbox
}%
\ifdim\wd\onelinebox>\myboxwidth
\hbox to \myboxwidth{%
$\Pi_{5,42}$ spans $L_{150.6}$%
\hfil
$22\infty22$ (shared)%
}%
\box\matricesbox
\else
\hbox to \myboxwidth{%
\unhbox\onelinebox
}%
\fi
\else
\hbox to \myboxwidth{%
$\Pi_{5,42}$ spans $L_{150.6}$%
\hfil}%
\hbox to \myboxwidth{%
$22\infty22$ (shared)%
\hfil}%
\box\matricesbox
\fi
}%
\hfill\discretionary{}{}{}%
\setbox\matricesbox=\hbox{%
{$\left[\!\llap{\phantom{%
\begingroup \smaller\smaller\smaller\begin{tabular}{@{}c@{}}%
\phantom{0}\\\phantom{0}\\\phantom{0}
\end{tabular}\endgroup%
}}\right.$}%
\begingroup \smaller\smaller\smaller\begin{tabular}{@{}c@{}}%
-1/16\\\phantom{0}\\\phantom{0}
\end{tabular}\endgroup%
\kern3pt%
\begingroup \smaller\smaller\smaller\begin{tabular}{@{}c@{}}%
\phantom{0}\\5\\-5/4
\end{tabular}\endgroup%
\kern3pt%
\begingroup \smaller\smaller\smaller\begin{tabular}{@{}c@{}}%
\phantom{0}\\-5/4\\105/16
\end{tabular}\endgroup%
{$\left.\llap{\phantom{%
\begingroup \smaller\smaller\smaller\begin{tabular}{@{}c@{}}%
\phantom{0}\\\phantom{0}\\\phantom{0}
\end{tabular}\endgroup%
}}\!\right]$}%
{$\left[\!\llap{\phantom{%
\begingroup \smaller\smaller\smaller\begin{tabular}{@{}c@{}}%
0\\0\\0
\end{tabular}\endgroup%
}}\right.$}%
\begingroup \smaller\smaller\smaller\begin{tabular}{@{}c@{}}%
4\\-1\\0
\end{tabular}\endgroup%
\kern3pt%
\begingroup \smaller\smaller\smaller\begin{tabular}{@{}c@{}}%
25\\-2\\-3
\end{tabular}\endgroup%
\kern3pt%
\begingroup \smaller\smaller\smaller\begin{tabular}{@{}c@{}}%
80\\2\\-8
\end{tabular}\endgroup%
\kern3pt%
\begingroup \smaller\smaller\smaller\begin{tabular}{@{}c@{}}%
20\\3\\0
\end{tabular}\endgroup%
\kern3pt%
\begingroup \smaller\smaller\smaller\begin{tabular}{@{}c@{}}%
5\\0\\1
\end{tabular}\endgroup%
{$\left.\llap{\phantom{%
\begingroup \smaller\smaller\smaller\begin{tabular}{@{}c@{}}%
0\\0\\0
\end{tabular}\endgroup%
}}\!\right]$}%
}%
\ifdim\wd\matricesbox>\halfwidth\myboxwidth=\hsize\else\myboxwidth=\halfwidth\fi
\vbox{%
\ifdim\myboxwidth=\hsize
\setbox\onelinebox=\hbox{%
\vbox{\hbox{%
$\Pi_{5,43}$ spans $L_{149.22}$%
}\hbox{%
$22\infty\infty2$ (shared)%
}%
}%
\hfill\copy\matricesbox
}%
\ifdim\wd\onelinebox>\myboxwidth
\hbox to \myboxwidth{%
$\Pi_{5,43}$ spans $L_{149.22}$%
\hfil
$22\infty\infty2$ (shared)%
}%
\box\matricesbox
\else
\hbox to \myboxwidth{%
\unhbox\onelinebox
}%
\fi
\else
\hbox to \myboxwidth{%
$\Pi_{5,43}$ spans $L_{149.22}$%
\hfil}%
\hbox to \myboxwidth{%
$22\infty\infty2$ (shared)%
\hfil}%
\box\matricesbox
\fi
}%
\hfill\discretionary{}{}{}%
\setbox\matricesbox=\hbox{%
{$\left[\!\llap{\phantom{%
\begingroup \smaller\smaller\smaller\begin{tabular}{@{}c@{}}%
\phantom{0}\\\phantom{0}\\\phantom{0}
\end{tabular}\endgroup%
}}\right.$}%
\begingroup \smaller\smaller\smaller\begin{tabular}{@{}c@{}}%
-1/19\\\phantom{0}\\\phantom{0}
\end{tabular}\endgroup%
\kern3pt%
\begingroup \smaller\smaller\smaller\begin{tabular}{@{}c@{}}%
\phantom{0}\\30/19\\-15/19
\end{tabular}\endgroup%
\kern3pt%
\begingroup \smaller\smaller\smaller\begin{tabular}{@{}c@{}}%
\phantom{0}\\-15/19\\150/19
\end{tabular}\endgroup%
{$\left.\llap{\phantom{%
\begingroup \smaller\smaller\smaller\begin{tabular}{@{}c@{}}%
\phantom{0}\\\phantom{0}\\\phantom{0}
\end{tabular}\endgroup%
}}\!\right]$}%
{$\left[\!\llap{\phantom{%
\begingroup \smaller\smaller\smaller\begin{tabular}{@{}c@{}}%
0\\0\\0
\end{tabular}\endgroup%
}}\right.$}%
\begingroup \smaller\smaller\smaller\begin{tabular}{@{}c@{}}%
6\\0\\-1
\end{tabular}\endgroup%
\kern3pt%
\begingroup \smaller\smaller\smaller\begin{tabular}{@{}c@{}}%
15\\-4\\-1
\end{tabular}\endgroup%
\kern3pt%
\begingroup \smaller\smaller\smaller\begin{tabular}{@{}c@{}}%
15\\-3\\1
\end{tabular}\endgroup%
\kern3pt%
\begingroup \smaller\smaller\smaller\begin{tabular}{@{}c@{}}%
30\\4\\3
\end{tabular}\endgroup%
\kern3pt%
\begingroup \smaller\smaller\smaller\begin{tabular}{@{}c@{}}%
5\\2\\0
\end{tabular}\endgroup%
{$\left.\llap{\phantom{%
\begingroup \smaller\smaller\smaller\begin{tabular}{@{}c@{}}%
0\\0\\0
\end{tabular}\endgroup%
}}\!\right]$}%
}%
\ifdim\wd\matricesbox>\halfwidth\myboxwidth=\hsize\else\myboxwidth=\halfwidth\fi
\vbox{%
\ifdim\myboxwidth=\hsize
\setbox\onelinebox=\hbox{%
\vbox{\hbox{%
$\Pi_{5,44}$ spans $L_{19.6}$%
}\hbox{%
$22422$ (shared)%
}%
}%
\hfill\copy\matricesbox
}%
\ifdim\wd\onelinebox>\myboxwidth
\hbox to \myboxwidth{%
$\Pi_{5,44}$ spans $L_{19.6}$%
\hfil
$22422$ (shared)%
}%
\box\matricesbox
\else
\hbox to \myboxwidth{%
\unhbox\onelinebox
}%
\fi
\else
\hbox to \myboxwidth{%
$\Pi_{5,44}$ spans $L_{19.6}$%
\hfil}%
\hbox to \myboxwidth{%
$22422$ (shared)%
\hfil}%
\box\matricesbox
\fi
}%
\hfill\discretionary{}{}{}%
\setbox\matricesbox=\hbox{%
{$\left[\!\llap{\phantom{%
\begingroup \smaller\smaller\smaller\begin{tabular}{@{}c@{}}%
\phantom{0}\\\phantom{0}\\\phantom{0}
\end{tabular}\endgroup%
}}\right.$}%
\begingroup \smaller\smaller\smaller\begin{tabular}{@{}c@{}}%
-1/3\\\phantom{0}\\\phantom{0}
\end{tabular}\endgroup%
\kern3pt%
\begingroup \smaller\smaller\smaller\begin{tabular}{@{}c@{}}%
\phantom{0}\\4/3\\-2/3
\end{tabular}\endgroup%
\kern3pt%
\begingroup \smaller\smaller\smaller\begin{tabular}{@{}c@{}}%
\phantom{0}\\-2/3\\10/3
\end{tabular}\endgroup%
{$\left.\llap{\phantom{%
\begingroup \smaller\smaller\smaller\begin{tabular}{@{}c@{}}%
\phantom{0}\\\phantom{0}\\\phantom{0}
\end{tabular}\endgroup%
}}\!\right]$}%
{$\left[\!\llap{\phantom{%
\begingroup \smaller\smaller\smaller\begin{tabular}{@{}c@{}}%
0\\0\\0
\end{tabular}\endgroup%
}}\right.$}%
\begingroup \smaller\smaller\smaller\begin{tabular}{@{}c@{}}%
1\\1\\0
\end{tabular}\endgroup%
\kern3pt%
\begingroup \smaller\smaller\smaller\begin{tabular}{@{}c@{}}%
2\\1\\1
\end{tabular}\endgroup%
\kern3pt%
\begingroup \smaller\smaller\smaller\begin{tabular}{@{}c@{}}%
3\\-1\\1
\end{tabular}\endgroup%
\kern3pt%
\begingroup \smaller\smaller\smaller\begin{tabular}{@{}c@{}}%
3\\-2\\-1
\end{tabular}\endgroup%
\kern3pt%
\begingroup \smaller\smaller\smaller\begin{tabular}{@{}c@{}}%
12\\1\\-4
\end{tabular}\endgroup%
{$\left.\llap{\phantom{%
\begingroup \smaller\smaller\smaller\begin{tabular}{@{}c@{}}%
0\\0\\0
\end{tabular}\endgroup%
}}\!\right]$}%
}%
\ifdim\wd\matricesbox>\halfwidth\myboxwidth=\hsize\else\myboxwidth=\halfwidth\fi
\vbox{%
\ifdim\myboxwidth=\hsize
\setbox\onelinebox=\hbox{%
\vbox{\hbox{%
$\Pi_{5,45}$ spans $L_{4.14}$%
}\hbox{%
$22\infty\infty2$%
}%
}%
\hfill\copy\matricesbox
}%
\ifdim\wd\onelinebox>\myboxwidth
\hbox to \myboxwidth{%
$\Pi_{5,45}$ spans $L_{4.14}$%
\hfil
$22\infty\infty2$%
}%
\box\matricesbox
\else
\hbox to \myboxwidth{%
\unhbox\onelinebox
}%
\fi
\else
\hbox to \myboxwidth{%
$\Pi_{5,45}$ spans $L_{4.14}$%
\hfil}%
\hbox to \myboxwidth{%
$22\infty\infty2$%
\hfil}%
\box\matricesbox
\fi
}%
\hfill\discretionary{}{}{}%
\setbox\matricesbox=\hbox{%
{$\left[\!\llap{\phantom{%
\begingroup \smaller\smaller\smaller\begin{tabular}{@{}c@{}}%
\phantom{0}\\\phantom{0}\\\phantom{0}
\end{tabular}\endgroup%
}}\right.$}%
\begingroup \smaller\smaller\smaller\begin{tabular}{@{}c@{}}%
-3/8\\\phantom{0}\\\phantom{0}
\end{tabular}\endgroup%
\kern3pt%
\begingroup \smaller\smaller\smaller\begin{tabular}{@{}c@{}}%
\phantom{0}\\7/2\\-1/2
\end{tabular}\endgroup%
\kern3pt%
\begingroup \smaller\smaller\smaller\begin{tabular}{@{}c@{}}%
\phantom{0}\\-1/2\\7/2
\end{tabular}\endgroup%
{$\left.\llap{\phantom{%
\begingroup \smaller\smaller\smaller\begin{tabular}{@{}c@{}}%
\phantom{0}\\\phantom{0}\\\phantom{0}
\end{tabular}\endgroup%
}}\!\right]$}%
{$\left[\!\llap{\phantom{%
\begingroup \smaller\smaller\smaller\begin{tabular}{@{}c@{}}%
0\\0\\0
\end{tabular}\endgroup%
}}\right.$}%
\begingroup \smaller\smaller\smaller\begin{tabular}{@{}c@{}}%
2\\0\\1
\end{tabular}\endgroup%
\kern3pt%
\begingroup \smaller\smaller\smaller\begin{tabular}{@{}c@{}}%
6\\-2\\1
\end{tabular}\endgroup%
\kern3pt%
\begingroup \smaller\smaller\smaller\begin{tabular}{@{}c@{}}%
8\\-3\\-1
\end{tabular}\endgroup%
\kern3pt%
\begingroup \smaller\smaller\smaller\begin{tabular}{@{}c@{}}%
2\\0\\-1
\end{tabular}\endgroup%
\kern3pt%
\begingroup \smaller\smaller\smaller\begin{tabular}{@{}c@{}}%
2\\1\\0
\end{tabular}\endgroup%
{$\left.\llap{\phantom{%
\begingroup \smaller\smaller\smaller\begin{tabular}{@{}c@{}}%
0\\0\\0
\end{tabular}\endgroup%
}}\!\right]$}%
}%
\ifdim\wd\matricesbox>\halfwidth\myboxwidth=\hsize\else\myboxwidth=\halfwidth\fi
\vbox{%
\ifdim\myboxwidth=\hsize
\setbox\onelinebox=\hbox{%
\vbox{\hbox{%
$\Pi_{5,46}$ spans $L_{7.9}$%
}\hbox{%
$22\infty3\infty$%
}%
}%
\hfill\copy\matricesbox
}%
\ifdim\wd\onelinebox>\myboxwidth
\hbox to \myboxwidth{%
$\Pi_{5,46}$ spans $L_{7.9}$%
\hfil
$22\infty3\infty$%
}%
\box\matricesbox
\else
\hbox to \myboxwidth{%
\unhbox\onelinebox
}%
\fi
\else
\hbox to \myboxwidth{%
$\Pi_{5,46}$ spans $L_{7.9}$%
\hfil}%
\hbox to \myboxwidth{%
$22\infty3\infty$%
\hfil}%
\box\matricesbox
\fi
}%
\hfill\discretionary{}{}{}%
\setbox\matricesbox=\hbox{%
{$\left[\!\llap{\phantom{%
\begingroup \smaller\smaller\smaller\begin{tabular}{@{}c@{}}%
\phantom{0}\\\phantom{0}\\\phantom{0}
\end{tabular}\endgroup%
}}\right.$}%
\begingroup \smaller\smaller\smaller\begin{tabular}{@{}c@{}}%
-1/14\\\phantom{0}\\\phantom{0}
\end{tabular}\endgroup%
\kern3pt%
\begingroup \smaller\smaller\smaller\begin{tabular}{@{}c@{}}%
\phantom{0}\\30/7\\-15/7
\end{tabular}\endgroup%
\kern3pt%
\begingroup \smaller\smaller\smaller\begin{tabular}{@{}c@{}}%
\phantom{0}\\-15/7\\60/7
\end{tabular}\endgroup%
{$\left.\llap{\phantom{%
\begingroup \smaller\smaller\smaller\begin{tabular}{@{}c@{}}%
\phantom{0}\\\phantom{0}\\\phantom{0}
\end{tabular}\endgroup%
}}\!\right]$}%
{$\left[\!\llap{\phantom{%
\begingroup \smaller\smaller\smaller\begin{tabular}{@{}c@{}}%
0\\0\\0
\end{tabular}\endgroup%
}}\right.$}%
\begingroup \smaller\smaller\smaller\begin{tabular}{@{}c@{}}%
6\\1\\1
\end{tabular}\endgroup%
\kern3pt%
\begingroup \smaller\smaller\smaller\begin{tabular}{@{}c@{}}%
10\\2\\0
\end{tabular}\endgroup%
\kern3pt%
\begingroup \smaller\smaller\smaller\begin{tabular}{@{}c@{}}%
30\\1\\-3
\end{tabular}\endgroup%
\kern3pt%
\begingroup \smaller\smaller\smaller\begin{tabular}{@{}c@{}}%
10\\-2\\-1
\end{tabular}\endgroup%
\kern3pt%
\begingroup \smaller\smaller\smaller\begin{tabular}{@{}c@{}}%
10\\-1\\1
\end{tabular}\endgroup%
{$\left.\llap{\phantom{%
\begingroup \smaller\smaller\smaller\begin{tabular}{@{}c@{}}%
0\\0\\0
\end{tabular}\endgroup%
}}\!\right]$}%
}%
\ifdim\wd\matricesbox>\halfwidth\myboxwidth=\hsize\else\myboxwidth=\halfwidth\fi
\vbox{%
\ifdim\myboxwidth=\hsize
\setbox\onelinebox=\hbox{%
\vbox{\hbox{%
$\Pi_{5,47}$ spans $L_{31.9}$%
}\hbox{%
$22632$%
}%
}%
\hfill\copy\matricesbox
}%
\ifdim\wd\onelinebox>\myboxwidth
\hbox to \myboxwidth{%
$\Pi_{5,47}$ spans $L_{31.9}$%
\hfil
$22632$%
}%
\box\matricesbox
\else
\hbox to \myboxwidth{%
\unhbox\onelinebox
}%
\fi
\else
\hbox to \myboxwidth{%
$\Pi_{5,47}$ spans $L_{31.9}$%
\hfil}%
\hbox to \myboxwidth{%
$22632$%
\hfil}%
\box\matricesbox
\fi
}%
\hfill\discretionary{}{}{}%
\setbox\matricesbox=\hbox{%
{$\left[\!\llap{\phantom{%
\begingroup \smaller\smaller\smaller\begin{tabular}{@{}c@{}}%
\phantom{0}\\\phantom{0}\\\phantom{0}
\end{tabular}\endgroup%
}}\right.$}%
\begingroup \smaller\smaller\smaller\begin{tabular}{@{}c@{}}%
-1/7\\\phantom{0}\\\phantom{0}
\end{tabular}\endgroup%
\kern3pt%
\begingroup \smaller\smaller\smaller\begin{tabular}{@{}c@{}}%
\phantom{0}\\18/7\\-6/7
\end{tabular}\endgroup%
\kern3pt%
\begingroup \smaller\smaller\smaller\begin{tabular}{@{}c@{}}%
\phantom{0}\\-6/7\\30/7
\end{tabular}\endgroup%
{$\left.\llap{\phantom{%
\begingroup \smaller\smaller\smaller\begin{tabular}{@{}c@{}}%
\phantom{0}\\\phantom{0}\\\phantom{0}
\end{tabular}\endgroup%
}}\!\right]$}%
{$\left[\!\llap{\phantom{%
\begingroup \smaller\smaller\smaller\begin{tabular}{@{}c@{}}%
0\\0\\0
\end{tabular}\endgroup%
}}\right.$}%
\begingroup \smaller\smaller\smaller\begin{tabular}{@{}c@{}}%
6\\2\\1
\end{tabular}\endgroup%
\kern3pt%
\begingroup \smaller\smaller\smaller\begin{tabular}{@{}c@{}}%
12\\3\\-1
\end{tabular}\endgroup%
\kern3pt%
\begingroup \smaller\smaller\smaller\begin{tabular}{@{}c@{}}%
8\\0\\-2
\end{tabular}\endgroup%
\kern3pt%
\begingroup \smaller\smaller\smaller\begin{tabular}{@{}c@{}}%
2\\-1\\0
\end{tabular}\endgroup%
\kern3pt%
\begingroup \smaller\smaller\smaller\begin{tabular}{@{}c@{}}%
3\\0\\1
\end{tabular}\endgroup%
{$\left.\llap{\phantom{%
\begingroup \smaller\smaller\smaller\begin{tabular}{@{}c@{}}%
0\\0\\0
\end{tabular}\endgroup%
}}\!\right]$}%
}%
\ifdim\wd\matricesbox>\halfwidth\myboxwidth=\hsize\else\myboxwidth=\halfwidth\fi
\vbox{%
\ifdim\myboxwidth=\hsize
\setbox\onelinebox=\hbox{%
\vbox{\hbox{%
$\Pi_{5,48}$ spans $L_{150.17}$%
}\hbox{%
$22\infty22$ (shared)%
}%
}%
\hfill\copy\matricesbox
}%
\ifdim\wd\onelinebox>\myboxwidth
\hbox to \myboxwidth{%
$\Pi_{5,48}$ spans $L_{150.17}$%
\hfil
$22\infty22$ (shared)%
}%
\box\matricesbox
\else
\hbox to \myboxwidth{%
\unhbox\onelinebox
}%
\fi
\else
\hbox to \myboxwidth{%
$\Pi_{5,48}$ spans $L_{150.17}$%
\hfil}%
\hbox to \myboxwidth{%
$22\infty22$ (shared)%
\hfil}%
\box\matricesbox
\fi
}%
\hfill\discretionary{}{}{}%
\setbox\matricesbox=\hbox{%
{$\left[\!\llap{\phantom{%
\begingroup \smaller\smaller\smaller\begin{tabular}{@{}c@{}}%
\phantom{0}\\\phantom{0}\\\phantom{0}
\end{tabular}\endgroup%
}}\right.$}%
\begingroup \smaller\smaller\smaller\begin{tabular}{@{}c@{}}%
-1/10\\\phantom{0}\\\phantom{0}
\end{tabular}\endgroup%
\kern3pt%
\begingroup \smaller\smaller\smaller\begin{tabular}{@{}c@{}}%
\phantom{0}\\11/10\\-1/2
\end{tabular}\endgroup%
\kern3pt%
\begingroup \smaller\smaller\smaller\begin{tabular}{@{}c@{}}%
\phantom{0}\\-1/2\\3/2
\end{tabular}\endgroup%
{$\left.\llap{\phantom{%
\begingroup \smaller\smaller\smaller\begin{tabular}{@{}c@{}}%
\phantom{0}\\\phantom{0}\\\phantom{0}
\end{tabular}\endgroup%
}}\!\right]$}%
{$\left[\!\llap{\phantom{%
\begingroup \smaller\smaller\smaller\begin{tabular}{@{}c@{}}%
0\\0\\0
\end{tabular}\endgroup%
}}\right.$}%
\begingroup \smaller\smaller\smaller\begin{tabular}{@{}c@{}}%
7\\-2\\-3
\end{tabular}\endgroup%
\kern3pt%
\begingroup \smaller\smaller\smaller\begin{tabular}{@{}c@{}}%
8\\2\\-2
\end{tabular}\endgroup%
\kern3pt%
\begingroup \smaller\smaller\smaller\begin{tabular}{@{}c@{}}%
14\\6\\2
\end{tabular}\endgroup%
\kern3pt%
\begingroup \smaller\smaller\smaller\begin{tabular}{@{}c@{}}%
14\\1\\5
\end{tabular}\endgroup%
\kern3pt%
\begingroup \smaller\smaller\smaller\begin{tabular}{@{}c@{}}%
1\\-1\\0
\end{tabular}\endgroup%
{$\left.\llap{\phantom{%
\begingroup \smaller\smaller\smaller\begin{tabular}{@{}c@{}}%
0\\0\\0
\end{tabular}\endgroup%
}}\!\right]$}%
}%
\ifdim\wd\matricesbox>\halfwidth\myboxwidth=\hsize\else\myboxwidth=\halfwidth\fi
\vbox{%
\ifdim\myboxwidth=\hsize
\setbox\onelinebox=\hbox{%
\vbox{\hbox{%
$\Pi_{5,49}$ spans $L_{9.5}$%
}\hbox{%
$22\infty22$%
}%
}%
\hfill\copy\matricesbox
}%
\ifdim\wd\onelinebox>\myboxwidth
\hbox to \myboxwidth{%
$\Pi_{5,49}$ spans $L_{9.5}$%
\hfil
$22\infty22$%
}%
\box\matricesbox
\else
\hbox to \myboxwidth{%
\unhbox\onelinebox
}%
\fi
\else
\hbox to \myboxwidth{%
$\Pi_{5,49}$ spans $L_{9.5}$%
\hfil}%
\hbox to \myboxwidth{%
$22\infty22$%
\hfil}%
\box\matricesbox
\fi
}%
\hfill\discretionary{}{}{}%
\setbox\matricesbox=\hbox{%
{$\left[\!\llap{\phantom{%
\begingroup \smaller\smaller\smaller\begin{tabular}{@{}c@{}}%
\phantom{0}\\\phantom{0}\\\phantom{0}
\end{tabular}\endgroup%
}}\right.$}%
\begingroup \smaller\smaller\smaller\begin{tabular}{@{}c@{}}%
-1/10\\\phantom{0}\\\phantom{0}
\end{tabular}\endgroup%
\kern3pt%
\begingroup \smaller\smaller\smaller\begin{tabular}{@{}c@{}}%
\phantom{0}\\12/5\\-3/5
\end{tabular}\endgroup%
\kern3pt%
\begingroup \smaller\smaller\smaller\begin{tabular}{@{}c@{}}%
\phantom{0}\\-3/5\\39/10
\end{tabular}\endgroup%
{$\left.\llap{\phantom{%
\begingroup \smaller\smaller\smaller\begin{tabular}{@{}c@{}}%
\phantom{0}\\\phantom{0}\\\phantom{0}
\end{tabular}\endgroup%
}}\!\right]$}%
{$\left[\!\llap{\phantom{%
\begingroup \smaller\smaller\smaller\begin{tabular}{@{}c@{}}%
0\\0\\0
\end{tabular}\endgroup%
}}\right.$}%
\begingroup \smaller\smaller\smaller\begin{tabular}{@{}c@{}}%
15\\-2\\-3
\end{tabular}\endgroup%
\kern3pt%
\begingroup \smaller\smaller\smaller\begin{tabular}{@{}c@{}}%
24\\2\\-4
\end{tabular}\endgroup%
\kern3pt%
\begingroup \smaller\smaller\smaller\begin{tabular}{@{}c@{}}%
6\\2\\0
\end{tabular}\endgroup%
\kern3pt%
\begingroup \smaller\smaller\smaller\begin{tabular}{@{}c@{}}%
3\\0\\1
\end{tabular}\endgroup%
\kern3pt%
\begingroup \smaller\smaller\smaller\begin{tabular}{@{}c@{}}%
2\\-1\\0
\end{tabular}\endgroup%
{$\left.\llap{\phantom{%
\begingroup \smaller\smaller\smaller\begin{tabular}{@{}c@{}}%
0\\0\\0
\end{tabular}\endgroup%
}}\!\right]$}%
}%
\ifdim\wd\matricesbox>\halfwidth\myboxwidth=\hsize\else\myboxwidth=\halfwidth\fi
\vbox{%
\ifdim\myboxwidth=\hsize
\setbox\onelinebox=\hbox{%
\vbox{\hbox{%
$\Pi_{5,50}$ spans $L_{127.6}$%
}\hbox{%
$22422$ (shared)%
}%
}%
\hfill\copy\matricesbox
}%
\ifdim\wd\onelinebox>\myboxwidth
\hbox to \myboxwidth{%
$\Pi_{5,50}$ spans $L_{127.6}$%
\hfil
$22422$ (shared)%
}%
\box\matricesbox
\else
\hbox to \myboxwidth{%
\unhbox\onelinebox
}%
\fi
\else
\hbox to \myboxwidth{%
$\Pi_{5,50}$ spans $L_{127.6}$%
\hfil}%
\hbox to \myboxwidth{%
$22422$ (shared)%
\hfil}%
\box\matricesbox
\fi
}%
\hfill\discretionary{}{}{}%
\setbox\matricesbox=\hbox{%
{$\left[\!\llap{\phantom{%
\begingroup \smaller\smaller\smaller\begin{tabular}{@{}c@{}}%
\phantom{0}\\\phantom{0}\\\phantom{0}
\end{tabular}\endgroup%
}}\right.$}%
\begingroup \smaller\smaller\smaller\begin{tabular}{@{}c@{}}%
-1/24\\\phantom{0}\\\phantom{0}
\end{tabular}\endgroup%
\kern3pt%
\begingroup \smaller\smaller\smaller\begin{tabular}{@{}c@{}}%
\phantom{0}\\13/6\\-2/3
\end{tabular}\endgroup%
\kern3pt%
\begingroup \smaller\smaller\smaller\begin{tabular}{@{}c@{}}%
\phantom{0}\\-2/3\\20/3
\end{tabular}\endgroup%
{$\left.\llap{\phantom{%
\begingroup \smaller\smaller\smaller\begin{tabular}{@{}c@{}}%
\phantom{0}\\\phantom{0}\\\phantom{0}
\end{tabular}\endgroup%
}}\!\right]$}%
{$\left[\!\llap{\phantom{%
\begingroup \smaller\smaller\smaller\begin{tabular}{@{}c@{}}%
0\\0\\0
\end{tabular}\endgroup%
}}\right.$}%
\begingroup \smaller\smaller\smaller\begin{tabular}{@{}c@{}}%
2\\1\\0
\end{tabular}\endgroup%
\kern3pt%
\begingroup \smaller\smaller\smaller\begin{tabular}{@{}c@{}}%
28\\2\\3
\end{tabular}\endgroup%
\kern3pt%
\begingroup \smaller\smaller\smaller\begin{tabular}{@{}c@{}}%
12\\-2\\1
\end{tabular}\endgroup%
\kern3pt%
\begingroup \smaller\smaller\smaller\begin{tabular}{@{}c@{}}%
14\\-3\\-1
\end{tabular}\endgroup%
\kern3pt%
\begingroup \smaller\smaller\smaller\begin{tabular}{@{}c@{}}%
16\\0\\-2
\end{tabular}\endgroup%
{$\left.\llap{\phantom{%
\begingroup \smaller\smaller\smaller\begin{tabular}{@{}c@{}}%
0\\0\\0
\end{tabular}\endgroup%
}}\!\right]$}%
}%
\ifdim\wd\matricesbox>\halfwidth\myboxwidth=\hsize\else\myboxwidth=\halfwidth\fi
\vbox{%
\ifdim\myboxwidth=\hsize
\setbox\onelinebox=\hbox{%
\vbox{\hbox{%
$\Pi_{5,51}$ spans $L_{25.3}$%
}\hbox{%
$22222$%
}%
}%
\hfill\copy\matricesbox
}%
\ifdim\wd\onelinebox>\myboxwidth
\hbox to \myboxwidth{%
$\Pi_{5,51}$ spans $L_{25.3}$%
\hfil
$22222$%
}%
\box\matricesbox
\else
\hbox to \myboxwidth{%
\unhbox\onelinebox
}%
\fi
\else
\hbox to \myboxwidth{%
$\Pi_{5,51}$ spans $L_{25.3}$%
\hfil}%
\hbox to \myboxwidth{%
$22222$%
\hfil}%
\box\matricesbox
\fi
}%
\hfill\discretionary{}{}{}%
\setbox\matricesbox=\hbox{%
{$\left[\!\llap{\phantom{%
\begingroup \smaller\smaller\smaller\begin{tabular}{@{}c@{}}%
\phantom{0}\\\phantom{0}\\\phantom{0}
\end{tabular}\endgroup%
}}\right.$}%
\begingroup \smaller\smaller\smaller\begin{tabular}{@{}c@{}}%
-1/11\\\phantom{0}\\\phantom{0}
\end{tabular}\endgroup%
\kern3pt%
\begingroup \smaller\smaller\smaller\begin{tabular}{@{}c@{}}%
\phantom{0}\\12/11\\-3/11
\end{tabular}\endgroup%
\kern3pt%
\begingroup \smaller\smaller\smaller\begin{tabular}{@{}c@{}}%
\phantom{0}\\-3/11\\42/11
\end{tabular}\endgroup%
{$\left.\llap{\phantom{%
\begingroup \smaller\smaller\smaller\begin{tabular}{@{}c@{}}%
\phantom{0}\\\phantom{0}\\\phantom{0}
\end{tabular}\endgroup%
}}\!\right]$}%
{$\left[\!\llap{\phantom{%
\begingroup \smaller\smaller\smaller\begin{tabular}{@{}c@{}}%
0\\0\\0
\end{tabular}\endgroup%
}}\right.$}%
\begingroup \smaller\smaller\smaller\begin{tabular}{@{}c@{}}%
1\\1\\0
\end{tabular}\endgroup%
\kern3pt%
\begingroup \smaller\smaller\smaller\begin{tabular}{@{}c@{}}%
15\\2\\3
\end{tabular}\endgroup%
\kern3pt%
\begingroup \smaller\smaller\smaller\begin{tabular}{@{}c@{}}%
6\\-2\\1
\end{tabular}\endgroup%
\kern3pt%
\begingroup \smaller\smaller\smaller\begin{tabular}{@{}c@{}}%
10\\-4\\-1
\end{tabular}\endgroup%
\kern3pt%
\begingroup \smaller\smaller\smaller\begin{tabular}{@{}c@{}}%
3\\0\\-1
\end{tabular}\endgroup%
{$\left.\llap{\phantom{%
\begingroup \smaller\smaller\smaller\begin{tabular}{@{}c@{}}%
0\\0\\0
\end{tabular}\endgroup%
}}\!\right]$}%
}%
\ifdim\wd\matricesbox>\halfwidth\myboxwidth=\hsize\else\myboxwidth=\halfwidth\fi
\vbox{%
\ifdim\myboxwidth=\hsize
\setbox\onelinebox=\hbox{%
\vbox{\hbox{%
$\Pi_{5,52}$ spans $L_{17.7}$%
}\hbox{%
$22222$%
}%
}%
\hfill\copy\matricesbox
}%
\ifdim\wd\onelinebox>\myboxwidth
\hbox to \myboxwidth{%
$\Pi_{5,52}$ spans $L_{17.7}$%
\hfil
$22222$%
}%
\box\matricesbox
\else
\hbox to \myboxwidth{%
\unhbox\onelinebox
}%
\fi
\else
\hbox to \myboxwidth{%
$\Pi_{5,52}$ spans $L_{17.7}$%
\hfil}%
\hbox to \myboxwidth{%
$22222$%
\hfil}%
\box\matricesbox
\fi
}%
\hfill\discretionary{}{}{}%
\setbox\matricesbox=\hbox{%
{$\left[\!\llap{\phantom{%
\begingroup \smaller\smaller\smaller\begin{tabular}{@{}c@{}}%
\phantom{0}\\\phantom{0}\\\phantom{0}
\end{tabular}\endgroup%
}}\right.$}%
\begingroup \smaller\smaller\smaller\begin{tabular}{@{}c@{}}%
-1/2\\\phantom{0}\\\phantom{0}
\end{tabular}\endgroup%
\kern3pt%
\begingroup \smaller\smaller\smaller\begin{tabular}{@{}c@{}}%
\phantom{0}\\3/2\\-1/2
\end{tabular}\endgroup%
\kern3pt%
\begingroup \smaller\smaller\smaller\begin{tabular}{@{}c@{}}%
\phantom{0}\\-1/2\\3/2
\end{tabular}\endgroup%
{$\left.\llap{\phantom{%
\begingroup \smaller\smaller\smaller\begin{tabular}{@{}c@{}}%
\phantom{0}\\\phantom{0}\\\phantom{0}
\end{tabular}\endgroup%
}}\!\right]$}%
{$\left[\!\llap{\phantom{%
\begingroup \smaller\smaller\smaller\begin{tabular}{@{}c@{}}%
0\\0\\0
\end{tabular}\endgroup%
}}\right.$}%
\begingroup \smaller\smaller\smaller\begin{tabular}{@{}c@{}}%
1\\1\\0
\end{tabular}\endgroup%
\kern3pt%
\begingroup \smaller\smaller\smaller\begin{tabular}{@{}c@{}}%
4\\-1\\-3
\end{tabular}\endgroup%
\kern3pt%
\begingroup \smaller\smaller\smaller\begin{tabular}{@{}c@{}}%
4\\-3\\-1
\end{tabular}\endgroup%
\kern3pt%
\begingroup \smaller\smaller\smaller\begin{tabular}{@{}c@{}}%
2\\-1\\1
\end{tabular}\endgroup%
\kern3pt%
\begingroup \smaller\smaller\smaller\begin{tabular}{@{}c@{}}%
4\\1\\3
\end{tabular}\endgroup%
{$\left.\llap{\phantom{%
\begingroup \smaller\smaller\smaller\begin{tabular}{@{}c@{}}%
0\\0\\0
\end{tabular}\endgroup%
}}\!\right]$}%
}%
\ifdim\wd\matricesbox>\halfwidth\myboxwidth=\hsize\else\myboxwidth=\halfwidth\fi
\vbox{%
\ifdim\myboxwidth=\hsize
\setbox\onelinebox=\hbox{%
\vbox{\hbox{%
$\Pi_{5,53}$ spans $L_{1.6}$%
}\hbox{%
$\infty\infty22\infty$%
}%
}%
\hfill\copy\matricesbox
}%
\ifdim\wd\onelinebox>\myboxwidth
\hbox to \myboxwidth{%
$\Pi_{5,53}$ spans $L_{1.6}$%
\hfil
$\infty\infty22\infty$%
}%
\box\matricesbox
\else
\hbox to \myboxwidth{%
\unhbox\onelinebox
}%
\fi
\else
\hbox to \myboxwidth{%
$\Pi_{5,53}$ spans $L_{1.6}$%
\hfil}%
\hbox to \myboxwidth{%
$\infty\infty22\infty$%
\hfil}%
\box\matricesbox
\fi
}%
\hfill\discretionary{}{}{}%

\vskip2pt\hrule\vskip2pt

\leavevmode\setbox\matricesbox=\hbox{%
{$\left[\!\llap{\phantom{%
\begingroup \smaller\smaller\smaller\begin{tabular}{@{}c@{}}%
\phantom{0}\\\phantom{0}\\\phantom{0}\\\phantom{0}
\end{tabular}\endgroup%
}}\right.$}%
\begingroup \smaller\smaller\smaller\begin{tabular}{@{}c@{}}%
-1\\\phantom{0}\\\phantom{0}\\\phantom{0}
\end{tabular}\endgroup%
\kern3pt%
\begingroup \smaller\smaller\smaller\begin{tabular}{@{}c@{}}%
\phantom{0}\\1\\\phantom{0}\\\phantom{0}
\end{tabular}\endgroup%
\kern3pt%
\begingroup \smaller\smaller\smaller\begin{tabular}{@{}c@{}}%
\phantom{0}\\\phantom{0}\\1\\\phantom{0}
\end{tabular}\endgroup%
\kern3pt%
\begingroup \smaller\smaller\smaller\begin{tabular}{@{}c@{}}%
\phantom{0}\\\phantom{0}\\\phantom{0}\\1
\end{tabular}\endgroup%
{$\left.\llap{\phantom{%
\begingroup \smaller\smaller\smaller\begin{tabular}{@{}c@{}}%
\phantom{0}\\\phantom{0}\\\phantom{0}\\\phantom{0}
\end{tabular}\endgroup%
}}\!\right]$}%
{$\left[\!\llap{\phantom{%
\begingroup \smaller\smaller\smaller\begin{tabular}{@{}c@{}}%
0\\0\\0\\0
\end{tabular}\endgroup%
}}\right.$}%
\begingroup \smaller\smaller\smaller\begin{tabular}{@{}c@{}}%
1\\-1\\1\\0
\end{tabular}\endgroup%
{$\left.\llap{\phantom{%
\begingroup \smaller\smaller\smaller\begin{tabular}{@{}c@{}}%
0\\0\\0\\0
\end{tabular}\endgroup%
}}\!\right]$}%
}%
\ifdim\wd\matricesbox>\halfwidth\myboxwidth=\hsize\else\myboxwidth=\halfwidth\fi
\vbox{%
\ifdim\myboxwidth=\hsize
\setbox\onelinebox=\hbox{%
\vbox{\hbox{%
$\Pi_{6,1}=B_3=\hbox{GN}_{37}$ spans $L_{3.2}$%
}\hbox{%
$|\slashtwo|\slashtwo|\slashtwo|\slashtwo|\slashtwo|\slashtwo\rtimes D_{12}$%
}%
}%
\hfill\copy\matricesbox
}%
\ifdim\wd\onelinebox>\myboxwidth
\hbox to \myboxwidth{%
$\Pi_{6,1}=B_3=\hbox{GN}_{37}$ spans $L_{3.2}$%
\hfil
$|\slashtwo|\slashtwo|\slashtwo|\slashtwo|\slashtwo|\slashtwo\rtimes D_{12}$%
}%
\box\matricesbox
\else
\hbox to \myboxwidth{%
\unhbox\onelinebox
}%
\fi
\else
\hbox to \myboxwidth{%
$\Pi_{6,1}=B_3=\hbox{GN}_{37}$ spans $L_{3.2}$%
\hfil}%
\hbox to \myboxwidth{%
$|\slashtwo|\slashtwo|\slashtwo|\slashtwo|\slashtwo|\slashtwo\rtimes D_{12}$%
\hfil}%
\box\matricesbox
\fi
}%
\hfill\discretionary{}{}{}%
\setbox\matricesbox=\hbox{%
{$\left[\!\llap{\phantom{%
\begingroup \smaller\smaller\smaller\begin{tabular}{@{}c@{}}%
\phantom{0}\\\phantom{0}\\\phantom{0}\\\phantom{0}
\end{tabular}\endgroup%
}}\right.$}%
\begingroup \smaller\smaller\smaller\begin{tabular}{@{}c@{}}%
-1\\\phantom{0}\\\phantom{0}\\\phantom{0}
\end{tabular}\endgroup%
\kern3pt%
\begingroup \smaller\smaller\smaller\begin{tabular}{@{}c@{}}%
\phantom{0}\\3\\\phantom{0}\\\phantom{0}
\end{tabular}\endgroup%
\kern3pt%
\begingroup \smaller\smaller\smaller\begin{tabular}{@{}c@{}}%
\phantom{0}\\\phantom{0}\\3\\\phantom{0}
\end{tabular}\endgroup%
\kern3pt%
\begingroup \smaller\smaller\smaller\begin{tabular}{@{}c@{}}%
\phantom{0}\\\phantom{0}\\\phantom{0}\\3
\end{tabular}\endgroup%
{$\left.\llap{\phantom{%
\begingroup \smaller\smaller\smaller\begin{tabular}{@{}c@{}}%
\phantom{0}\\\phantom{0}\\\phantom{0}\\\phantom{0}
\end{tabular}\endgroup%
}}\!\right]$}%
{$\left[\!\llap{\phantom{%
\begingroup \smaller\smaller\smaller\begin{tabular}{@{}c@{}}%
0\\0\\0\\0
\end{tabular}\endgroup%
}}\right.$}%
\begingroup \smaller\smaller\smaller\begin{tabular}{@{}c@{}}%
2\\-1\\0\\1
\end{tabular}\endgroup%
{$\left.\llap{\phantom{%
\begingroup \smaller\smaller\smaller\begin{tabular}{@{}c@{}}%
0\\0\\0\\0
\end{tabular}\endgroup%
}}\!\right]$}%
}%
\ifdim\wd\matricesbox>\halfwidth\myboxwidth=\hsize\else\myboxwidth=\halfwidth\fi
\vbox{%
\ifdim\myboxwidth=\hsize
\setbox\onelinebox=\hbox{%
\vbox{\hbox{%
$\Pi_{6,2}=B_4=\hbox{GN}_{43}$ spans $L_{155.1}$%
}\hbox{%
$|\slashthree|\slashthree|\slashthree|\slashthree|\slashthree|\slashthree\rtimes D_{12}$%
}%
}%
\hfill\copy\matricesbox
}%
\ifdim\wd\onelinebox>\myboxwidth
\hbox to \myboxwidth{%
$\Pi_{6,2}=B_4=\hbox{GN}_{43}$ spans $L_{155.1}$%
\hfil
$|\slashthree|\slashthree|\slashthree|\slashthree|\slashthree|\slashthree\rtimes D_{12}$%
}%
\box\matricesbox
\else
\hbox to \myboxwidth{%
\unhbox\onelinebox
}%
\fi
\else
\hbox to \myboxwidth{%
$\Pi_{6,2}=B_4=\hbox{GN}_{43}$ spans $L_{155.1}$%
\hfil}%
\hbox to \myboxwidth{%
$|\slashthree|\slashthree|\slashthree|\slashthree|\slashthree|\slashthree\rtimes D_{12}$%
\hfil}%
\box\matricesbox
\fi
}%
\hfill\discretionary{}{}{}%
\setbox\matricesbox=\hbox{%
{$\left[\!\llap{\phantom{%
\begingroup \smaller\smaller\smaller\begin{tabular}{@{}c@{}}%
\phantom{0}\\\phantom{0}\\\phantom{0}\\\phantom{0}
\end{tabular}\endgroup%
}}\right.$}%
\begingroup \smaller\smaller\smaller\begin{tabular}{@{}c@{}}%
-3\\\phantom{0}\\\phantom{0}\\\phantom{0}
\end{tabular}\endgroup%
\kern3pt%
\begingroup \smaller\smaller\smaller\begin{tabular}{@{}c@{}}%
\phantom{0}\\2\\\phantom{0}\\\phantom{0}
\end{tabular}\endgroup%
\kern3pt%
\begingroup \smaller\smaller\smaller\begin{tabular}{@{}c@{}}%
\phantom{0}\\\phantom{0}\\2\\\phantom{0}
\end{tabular}\endgroup%
\kern3pt%
\begingroup \smaller\smaller\smaller\begin{tabular}{@{}c@{}}%
\phantom{0}\\\phantom{0}\\\phantom{0}\\2
\end{tabular}\endgroup%
{$\left.\llap{\phantom{%
\begingroup \smaller\smaller\smaller\begin{tabular}{@{}c@{}}%
\phantom{0}\\\phantom{0}\\\phantom{0}\\\phantom{0}
\end{tabular}\endgroup%
}}\!\right]$}%
{$\left[\!\llap{\phantom{%
\begingroup \smaller\smaller\smaller\begin{tabular}{@{}c@{}}%
0\\0\\0\\0
\end{tabular}\endgroup%
}}\right.$}%
\begingroup \smaller\smaller\smaller\begin{tabular}{@{}c@{}}%
1\\-1\\0\\1
\end{tabular}\endgroup%
{$\left.\llap{\phantom{%
\begingroup \smaller\smaller\smaller\begin{tabular}{@{}c@{}}%
0\\0\\0\\0
\end{tabular}\endgroup%
}}\!\right]$}%
}%
\ifdim\wd\matricesbox>\halfwidth\myboxwidth=\hsize\else\myboxwidth=\halfwidth\fi
\vbox{%
\ifdim\myboxwidth=\hsize
\setbox\onelinebox=\hbox{%
\vbox{\hbox{%
$\Pi_{6,3}=A_{3,II}=\hbox{GN}_{46}$ spans $L_{7.7}$%
}\hbox{%
$|\slashinfty|\slashinfty|\slashinfty|\slashinfty|\slashinfty|\slashinfty\rtimes D_{12}$%
}%
}%
\hfill\copy\matricesbox
}%
\ifdim\wd\onelinebox>\myboxwidth
\hbox to \myboxwidth{%
$\Pi_{6,3}=A_{3,II}=\hbox{GN}_{46}$ spans $L_{7.7}$%
\hfil
$|\slashinfty|\slashinfty|\slashinfty|\slashinfty|\slashinfty|\slashinfty\rtimes D_{12}$%
}%
\box\matricesbox
\else
\hbox to \myboxwidth{%
\unhbox\onelinebox
}%
\fi
\else
\hbox to \myboxwidth{%
$\Pi_{6,3}=A_{3,II}=\hbox{GN}_{46}$ spans $L_{7.7}$%
\hfil}%
\hbox to \myboxwidth{%
$|\slashinfty|\slashinfty|\slashinfty|\slashinfty|\slashinfty|\slashinfty\rtimes D_{12}$%
\hfil}%
\box\matricesbox
\fi
}%
\hfill\discretionary{}{}{}%
\setbox\matricesbox=\hbox{%
{$\left[\!\llap{\phantom{%
\begingroup \smaller\smaller\smaller\begin{tabular}{@{}c@{}}%
\phantom{0}\\\phantom{0}\\\phantom{0}\\\phantom{0}
\end{tabular}\endgroup%
}}\right.$}%
\begingroup \smaller\smaller\smaller\begin{tabular}{@{}c@{}}%
-1/3\\\phantom{0}\\\phantom{0}\\\phantom{0}
\end{tabular}\endgroup%
\kern3pt%
\begingroup \smaller\smaller\smaller\begin{tabular}{@{}c@{}}%
\phantom{0}\\5/9\\\phantom{0}\\\phantom{0}
\end{tabular}\endgroup%
\kern3pt%
\begingroup \smaller\smaller\smaller\begin{tabular}{@{}c@{}}%
\phantom{0}\\\phantom{0}\\5/9\\\phantom{0}
\end{tabular}\endgroup%
\kern3pt%
\begingroup \smaller\smaller\smaller\begin{tabular}{@{}c@{}}%
\phantom{0}\\\phantom{0}\\\phantom{0}\\5/9
\end{tabular}\endgroup%
{$\left.\llap{\phantom{%
\begingroup \smaller\smaller\smaller\begin{tabular}{@{}c@{}}%
\phantom{0}\\\phantom{0}\\\phantom{0}\\\phantom{0}
\end{tabular}\endgroup%
}}\!\right]$}%
{$\left[\!\llap{\phantom{%
\begingroup \smaller\smaller\smaller\begin{tabular}{@{}c@{}}%
0\\0\\0\\0
\end{tabular}\endgroup%
}}\right.$}%
\begingroup \smaller\smaller\smaller\begin{tabular}{@{}c@{}}%
2\\2\\-1\\-1
\end{tabular}\endgroup%
\kern3pt%
\begingroup \smaller\smaller\smaller\begin{tabular}{@{}c@{}}%
5\\2\\2\\-4
\end{tabular}\endgroup%
{$\left.\llap{\phantom{%
\begingroup \smaller\smaller\smaller\begin{tabular}{@{}c@{}}%
0\\0\\0\\0
\end{tabular}\endgroup%
}}\!\right]$}%
}%
\ifdim\wd\matricesbox>\halfwidth\myboxwidth=\hsize\else\myboxwidth=\halfwidth\fi
\vbox{%
\ifdim\myboxwidth=\hsize
\setbox\onelinebox=\hbox{%
\vbox{\hbox{%
$\Pi_{6,4}$ spans $L_{6.3}$%
}\hbox{%
$|2|2|2|2|2|2\rtimes D_{6}$%
}%
}%
\hfill\copy\matricesbox
}%
\ifdim\wd\onelinebox>\myboxwidth
\hbox to \myboxwidth{%
$\Pi_{6,4}$ spans $L_{6.3}$%
\hfil
$|2|2|2|2|2|2\rtimes D_{6}$%
}%
\box\matricesbox
\else
\hbox to \myboxwidth{%
\unhbox\onelinebox
}%
\fi
\else
\hbox to \myboxwidth{%
$\Pi_{6,4}$ spans $L_{6.3}$%
\hfil}%
\hbox to \myboxwidth{%
$|2|2|2|2|2|2\rtimes D_{6}$%
\hfil}%
\box\matricesbox
\fi
}%
\hfill\discretionary{}{}{}%
\setbox\matricesbox=\hbox{%
{$\left[\!\llap{\phantom{%
\begingroup \smaller\smaller\smaller\begin{tabular}{@{}c@{}}%
\phantom{0}\\\phantom{0}\\\phantom{0}
\end{tabular}\endgroup%
}}\right.$}%
\begingroup \smaller\smaller\smaller\begin{tabular}{@{}c@{}}%
-1/8\\\phantom{0}\\\phantom{0}
\end{tabular}\endgroup%
\kern3pt%
\begingroup \smaller\smaller\smaller\begin{tabular}{@{}c@{}}%
\phantom{0}\\3\\\phantom{0}
\end{tabular}\endgroup%
\kern3pt%
\begingroup \smaller\smaller\smaller\begin{tabular}{@{}c@{}}%
\phantom{0}\\\phantom{0}\\15/2
\end{tabular}\endgroup%
{$\left.\llap{\phantom{%
\begingroup \smaller\smaller\smaller\begin{tabular}{@{}c@{}}%
\phantom{0}\\\phantom{0}\\\phantom{0}
\end{tabular}\endgroup%
}}\!\right]$}%
{$\left[\!\llap{\phantom{%
\begingroup \smaller\smaller\smaller\begin{tabular}{@{}c@{}}%
0\\0\\0
\end{tabular}\endgroup%
}}\right.$}%
\begingroup \smaller\smaller\smaller\begin{tabular}{@{}c@{}}%
16\\-4\\0
\end{tabular}\endgroup%
\kern3pt%
\begingroup \smaller\smaller\smaller\begin{tabular}{@{}c@{}}%
6\\-1\\-1
\end{tabular}\endgroup%
{$\left.\llap{\phantom{%
\begingroup \smaller\smaller\smaller\begin{tabular}{@{}c@{}}%
0\\0\\0
\end{tabular}\endgroup%
}}\!\right]$}%
}%
\ifdim\wd\matricesbox>\halfwidth\myboxwidth=\hsize\else\myboxwidth=\halfwidth\fi
\vbox{%
\ifdim\myboxwidth=\hsize
\setbox\onelinebox=\hbox{%
\vbox{\hbox{%
$\Pi_{6,5}$ spans $L_{19.5}$%
}\hbox{%
$2|2\slashtwo2|2\slashtwo\rtimes D_{4}$%
}%
}%
\hfill\copy\matricesbox
}%
\ifdim\wd\onelinebox>\myboxwidth
\hbox to \myboxwidth{%
$\Pi_{6,5}$ spans $L_{19.5}$%
\hfil
$2|2\slashtwo2|2\slashtwo\rtimes D_{4}$%
}%
\box\matricesbox
\else
\hbox to \myboxwidth{%
\unhbox\onelinebox
}%
\fi
\else
\hbox to \myboxwidth{%
$\Pi_{6,5}$ spans $L_{19.5}$%
\hfil}%
\hbox to \myboxwidth{%
$2|2\slashtwo2|2\slashtwo\rtimes D_{4}$%
\hfil}%
\box\matricesbox
\fi
}%
\hfill\discretionary{}{}{}%
\setbox\matricesbox=\hbox{%
{$\left[\!\llap{\phantom{%
\begingroup \smaller\smaller\smaller\begin{tabular}{@{}c@{}}%
\phantom{0}\\\phantom{0}\\\phantom{0}
\end{tabular}\endgroup%
}}\right.$}%
\begingroup \smaller\smaller\smaller\begin{tabular}{@{}c@{}}%
-1/2\\\phantom{0}\\\phantom{0}
\end{tabular}\endgroup%
\kern3pt%
\begingroup \smaller\smaller\smaller\begin{tabular}{@{}c@{}}%
\phantom{0}\\6\\\phantom{0}
\end{tabular}\endgroup%
\kern3pt%
\begingroup \smaller\smaller\smaller\begin{tabular}{@{}c@{}}%
\phantom{0}\\\phantom{0}\\3/2
\end{tabular}\endgroup%
{$\left.\llap{\phantom{%
\begingroup \smaller\smaller\smaller\begin{tabular}{@{}c@{}}%
\phantom{0}\\\phantom{0}\\\phantom{0}
\end{tabular}\endgroup%
}}\!\right]$}%
{$\left[\!\llap{\phantom{%
\begingroup \smaller\smaller\smaller\begin{tabular}{@{}c@{}}%
0\\0\\0
\end{tabular}\endgroup%
}}\right.$}%
\begingroup \smaller\smaller\smaller\begin{tabular}{@{}c@{}}%
3\\-1\\1
\end{tabular}\endgroup%
\kern3pt%
\begingroup \smaller\smaller\smaller\begin{tabular}{@{}c@{}}%
1\\0\\1
\end{tabular}\endgroup%
{$\left.\llap{\phantom{%
\begingroup \smaller\smaller\smaller\begin{tabular}{@{}c@{}}%
0\\0\\0
\end{tabular}\endgroup%
}}\!\right]$}%
}%
\ifdim\wd\matricesbox>\halfwidth\myboxwidth=\hsize\else\myboxwidth=\halfwidth\fi
\vbox{%
\ifdim\myboxwidth=\hsize
\setbox\onelinebox=\hbox{%
\vbox{\hbox{%
$\Pi_{6,6}$ spans $L_{123.6}$%
}\hbox{%
$\slashtwo2|2\slashtwo2|2\rtimes D_{4}$%
}%
}%
\hfill\copy\matricesbox
}%
\ifdim\wd\onelinebox>\myboxwidth
\hbox to \myboxwidth{%
$\Pi_{6,6}$ spans $L_{123.6}$%
\hfil
$\slashtwo2|2\slashtwo2|2\rtimes D_{4}$%
}%
\box\matricesbox
\else
\hbox to \myboxwidth{%
\unhbox\onelinebox
}%
\fi
\else
\hbox to \myboxwidth{%
$\Pi_{6,6}$ spans $L_{123.6}$%
\hfil}%
\hbox to \myboxwidth{%
$\slashtwo2|2\slashtwo2|2\rtimes D_{4}$%
\hfil}%
\box\matricesbox
\fi
}%
\hfill\discretionary{}{}{}%
\setbox\matricesbox=\hbox{%
{$\left[\!\llap{\phantom{%
\begingroup \smaller\smaller\smaller\begin{tabular}{@{}c@{}}%
\phantom{0}\\\phantom{0}\\\phantom{0}
\end{tabular}\endgroup%
}}\right.$}%
\begingroup \smaller\smaller\smaller\begin{tabular}{@{}c@{}}%
-1/8\\\phantom{0}\\\phantom{0}
\end{tabular}\endgroup%
\kern3pt%
\begingroup \smaller\smaller\smaller\begin{tabular}{@{}c@{}}%
\phantom{0}\\60\\\phantom{0}
\end{tabular}\endgroup%
\kern3pt%
\begingroup \smaller\smaller\smaller\begin{tabular}{@{}c@{}}%
\phantom{0}\\\phantom{0}\\5/2
\end{tabular}\endgroup%
{$\left.\llap{\phantom{%
\begingroup \smaller\smaller\smaller\begin{tabular}{@{}c@{}}%
\phantom{0}\\\phantom{0}\\\phantom{0}
\end{tabular}\endgroup%
}}\!\right]$}%
{$\left[\!\llap{\phantom{%
\begingroup \smaller\smaller\smaller\begin{tabular}{@{}c@{}}%
0\\0\\0
\end{tabular}\endgroup%
}}\right.$}%
\begingroup \smaller\smaller\smaller\begin{tabular}{@{}c@{}}%
20\\-1\\-2
\end{tabular}\endgroup%
\kern3pt%
\begingroup \smaller\smaller\smaller\begin{tabular}{@{}c@{}}%
2\\0\\-1
\end{tabular}\endgroup%
{$\left.\llap{\phantom{%
\begingroup \smaller\smaller\smaller\begin{tabular}{@{}c@{}}%
0\\0\\0
\end{tabular}\endgroup%
}}\!\right]$}%
}%
\ifdim\wd\matricesbox>\halfwidth\myboxwidth=\hsize\else\myboxwidth=\halfwidth\fi
\vbox{%
\ifdim\myboxwidth=\hsize
\setbox\onelinebox=\hbox{%
\vbox{\hbox{%
$\Pi_{6,7}$ spans $L_{19.8}$%
}\hbox{%
$\slashtwo2|2\slashtwo2|2\rtimes D_{4}$%
}%
}%
\hfill\copy\matricesbox
}%
\ifdim\wd\onelinebox>\myboxwidth
\hbox to \myboxwidth{%
$\Pi_{6,7}$ spans $L_{19.8}$%
\hfil
$\slashtwo2|2\slashtwo2|2\rtimes D_{4}$%
}%
\box\matricesbox
\else
\hbox to \myboxwidth{%
\unhbox\onelinebox
}%
\fi
\else
\hbox to \myboxwidth{%
$\Pi_{6,7}$ spans $L_{19.8}$%
\hfil}%
\hbox to \myboxwidth{%
$\slashtwo2|2\slashtwo2|2\rtimes D_{4}$%
\hfil}%
\box\matricesbox
\fi
}%
\hfill\discretionary{}{}{}%
\setbox\matricesbox=\hbox{%
{$\left[\!\llap{\phantom{%
\begingroup \smaller\smaller\smaller\begin{tabular}{@{}c@{}}%
\phantom{0}\\\phantom{0}\\\phantom{0}
\end{tabular}\endgroup%
}}\right.$}%
\begingroup \smaller\smaller\smaller\begin{tabular}{@{}c@{}}%
-1/2\\\phantom{0}\\\phantom{0}
\end{tabular}\endgroup%
\kern3pt%
\begingroup \smaller\smaller\smaller\begin{tabular}{@{}c@{}}%
\phantom{0}\\3/2\\\phantom{0}
\end{tabular}\endgroup%
\kern3pt%
\begingroup \smaller\smaller\smaller\begin{tabular}{@{}c@{}}%
\phantom{0}\\\phantom{0}\\5/2
\end{tabular}\endgroup%
{$\left.\llap{\phantom{%
\begingroup \smaller\smaller\smaller\begin{tabular}{@{}c@{}}%
\phantom{0}\\\phantom{0}\\\phantom{0}
\end{tabular}\endgroup%
}}\!\right]$}%
{$\left[\!\llap{\phantom{%
\begingroup \smaller\smaller\smaller\begin{tabular}{@{}c@{}}%
0\\0\\0
\end{tabular}\endgroup%
}}\right.$}%
\begingroup \smaller\smaller\smaller\begin{tabular}{@{}c@{}}%
6\\4\\0
\end{tabular}\endgroup%
\kern3pt%
\begingroup \smaller\smaller\smaller\begin{tabular}{@{}c@{}}%
2\\1\\1
\end{tabular}\endgroup%
{$\left.\llap{\phantom{%
\begingroup \smaller\smaller\smaller\begin{tabular}{@{}c@{}}%
0\\0\\0
\end{tabular}\endgroup%
}}\!\right]$}%
}%
\ifdim\wd\matricesbox>\halfwidth\myboxwidth=\hsize\else\myboxwidth=\halfwidth\fi
\vbox{%
\ifdim\myboxwidth=\hsize
\setbox\onelinebox=\hbox{%
\vbox{\hbox{%
$\Pi_{6,8}$ spans $L_{31.2}$%
}\hbox{%
$2|2\slashthree2|2\slashthree\rtimes D_{4}$ (shared)%
}%
}%
\hfill\copy\matricesbox
}%
\ifdim\wd\onelinebox>\myboxwidth
\hbox to \myboxwidth{%
$\Pi_{6,8}$ spans $L_{31.2}$%
\hfil
$2|2\slashthree2|2\slashthree\rtimes D_{4}$ (shared)%
}%
\box\matricesbox
\else
\hbox to \myboxwidth{%
\unhbox\onelinebox
}%
\fi
\else
\hbox to \myboxwidth{%
$\Pi_{6,8}$ spans $L_{31.2}$%
\hfil}%
\hbox to \myboxwidth{%
$2|2\slashthree2|2\slashthree\rtimes D_{4}$ (shared)%
\hfil}%
\box\matricesbox
\fi
}%
\hfill\discretionary{}{}{}%
\setbox\matricesbox=\hbox{%
{$\left[\!\llap{\phantom{%
\begingroup \smaller\smaller\smaller\begin{tabular}{@{}c@{}}%
\phantom{0}\\\phantom{0}\\\phantom{0}
\end{tabular}\endgroup%
}}\right.$}%
\begingroup \smaller\smaller\smaller\begin{tabular}{@{}c@{}}%
-1/8\\\phantom{0}\\\phantom{0}
\end{tabular}\endgroup%
\kern3pt%
\begingroup \smaller\smaller\smaller\begin{tabular}{@{}c@{}}%
\phantom{0}\\10\\\phantom{0}
\end{tabular}\endgroup%
\kern3pt%
\begingroup \smaller\smaller\smaller\begin{tabular}{@{}c@{}}%
\phantom{0}\\\phantom{0}\\25/2
\end{tabular}\endgroup%
{$\left.\llap{\phantom{%
\begingroup \smaller\smaller\smaller\begin{tabular}{@{}c@{}}%
\phantom{0}\\\phantom{0}\\\phantom{0}
\end{tabular}\endgroup%
}}\!\right]$}%
{$\left[\!\llap{\phantom{%
\begingroup \smaller\smaller\smaller\begin{tabular}{@{}c@{}}%
0\\0\\0
\end{tabular}\endgroup%
}}\right.$}%
\begingroup \smaller\smaller\smaller\begin{tabular}{@{}c@{}}%
32\\4\\0
\end{tabular}\endgroup%
\kern3pt%
\begingroup \smaller\smaller\smaller\begin{tabular}{@{}c@{}}%
10\\1\\1
\end{tabular}\endgroup%
{$\left.\llap{\phantom{%
\begingroup \smaller\smaller\smaller\begin{tabular}{@{}c@{}}%
0\\0\\0
\end{tabular}\endgroup%
}}\!\right]$}%
}%
\ifdim\wd\matricesbox>\halfwidth\myboxwidth=\hsize\else\myboxwidth=\halfwidth\fi
\vbox{%
\ifdim\myboxwidth=\hsize
\setbox\onelinebox=\hbox{%
\vbox{\hbox{%
$\Pi_{6,9}$ spans $L_{10.1}$%
}\hbox{%
$2|2\slashinfty2|2\slashinfty\rtimes D_{4}$%
}%
}%
\hfill\copy\matricesbox
}%
\ifdim\wd\onelinebox>\myboxwidth
\hbox to \myboxwidth{%
$\Pi_{6,9}$ spans $L_{10.1}$%
\hfil
$2|2\slashinfty2|2\slashinfty\rtimes D_{4}$%
}%
\box\matricesbox
\else
\hbox to \myboxwidth{%
\unhbox\onelinebox
}%
\fi
\else
\hbox to \myboxwidth{%
$\Pi_{6,9}$ spans $L_{10.1}$%
\hfil}%
\hbox to \myboxwidth{%
$2|2\slashinfty2|2\slashinfty\rtimes D_{4}$%
\hfil}%
\box\matricesbox
\fi
}%
\hfill\discretionary{}{}{}%
\setbox\matricesbox=\hbox{%
{$\left[\!\llap{\phantom{%
\begingroup \smaller\smaller\smaller\begin{tabular}{@{}c@{}}%
\phantom{0}\\\phantom{0}\\\phantom{0}
\end{tabular}\endgroup%
}}\right.$}%
\begingroup \smaller\smaller\smaller\begin{tabular}{@{}c@{}}%
-1/8\\\phantom{0}\\\phantom{0}
\end{tabular}\endgroup%
\kern3pt%
\begingroup \smaller\smaller\smaller\begin{tabular}{@{}c@{}}%
\phantom{0}\\15/2\\\phantom{0}
\end{tabular}\endgroup%
\kern3pt%
\begingroup \smaller\smaller\smaller\begin{tabular}{@{}c@{}}%
\phantom{0}\\\phantom{0}\\15
\end{tabular}\endgroup%
{$\left.\llap{\phantom{%
\begingroup \smaller\smaller\smaller\begin{tabular}{@{}c@{}}%
\phantom{0}\\\phantom{0}\\\phantom{0}
\end{tabular}\endgroup%
}}\!\right]$}%
{$\left[\!\llap{\phantom{%
\begingroup \smaller\smaller\smaller\begin{tabular}{@{}c@{}}%
0\\0\\0
\end{tabular}\endgroup%
}}\right.$}%
\begingroup \smaller\smaller\smaller\begin{tabular}{@{}c@{}}%
12\\-2\\0
\end{tabular}\endgroup%
\kern3pt%
\begingroup \smaller\smaller\smaller\begin{tabular}{@{}c@{}}%
10\\-1\\-1
\end{tabular}\endgroup%
{$\left.\llap{\phantom{%
\begingroup \smaller\smaller\smaller\begin{tabular}{@{}c@{}}%
0\\0\\0
\end{tabular}\endgroup%
}}\!\right]$}%
}%
\ifdim\wd\matricesbox>\halfwidth\myboxwidth=\hsize\else\myboxwidth=\halfwidth\fi
\vbox{%
\ifdim\myboxwidth=\hsize
\setbox\onelinebox=\hbox{%
\vbox{\hbox{%
$\Pi_{6,10}$ spans $L_{16.16}$%
}\hbox{%
$2|2\slashthree2|2\slashthree\rtimes D_{4}$%
}%
}%
\hfill\copy\matricesbox
}%
\ifdim\wd\onelinebox>\myboxwidth
\hbox to \myboxwidth{%
$\Pi_{6,10}$ spans $L_{16.16}$%
\hfil
$2|2\slashthree2|2\slashthree\rtimes D_{4}$%
}%
\box\matricesbox
\else
\hbox to \myboxwidth{%
\unhbox\onelinebox
}%
\fi
\else
\hbox to \myboxwidth{%
$\Pi_{6,10}$ spans $L_{16.16}$%
\hfil}%
\hbox to \myboxwidth{%
$2|2\slashthree2|2\slashthree\rtimes D_{4}$%
\hfil}%
\box\matricesbox
\fi
}%
\hfill\discretionary{}{}{}%
\setbox\matricesbox=\hbox{%
{$\left[\!\llap{\phantom{%
\begingroup \smaller\smaller\smaller\begin{tabular}{@{}c@{}}%
\phantom{0}\\\phantom{0}\\\phantom{0}
\end{tabular}\endgroup%
}}\right.$}%
\begingroup \smaller\smaller\smaller\begin{tabular}{@{}c@{}}%
-1/2\\\phantom{0}\\\phantom{0}
\end{tabular}\endgroup%
\kern3pt%
\begingroup \smaller\smaller\smaller\begin{tabular}{@{}c@{}}%
\phantom{0}\\3\\\phantom{0}
\end{tabular}\endgroup%
\kern3pt%
\begingroup \smaller\smaller\smaller\begin{tabular}{@{}c@{}}%
\phantom{0}\\\phantom{0}\\9/2
\end{tabular}\endgroup%
{$\left.\llap{\phantom{%
\begingroup \smaller\smaller\smaller\begin{tabular}{@{}c@{}}%
\phantom{0}\\\phantom{0}\\\phantom{0}
\end{tabular}\endgroup%
}}\!\right]$}%
{$\left[\!\llap{\phantom{%
\begingroup \smaller\smaller\smaller\begin{tabular}{@{}c@{}}%
0\\0\\0
\end{tabular}\endgroup%
}}\right.$}%
\begingroup \smaller\smaller\smaller\begin{tabular}{@{}c@{}}%
4\\-2\\0
\end{tabular}\endgroup%
\kern3pt%
\begingroup \smaller\smaller\smaller\begin{tabular}{@{}c@{}}%
3\\-1\\-1
\end{tabular}\endgroup%
{$\left.\llap{\phantom{%
\begingroup \smaller\smaller\smaller\begin{tabular}{@{}c@{}}%
0\\0\\0
\end{tabular}\endgroup%
}}\!\right]$}%
}%
\ifdim\wd\matricesbox>\halfwidth\myboxwidth=\hsize\else\myboxwidth=\halfwidth\fi
\vbox{%
\ifdim\myboxwidth=\hsize
\setbox\onelinebox=\hbox{%
\vbox{\hbox{%
$\Pi_{6,11}$ spans $L_{4.12}$%
}\hbox{%
$2|2\slashinfty2|2\slashinfty\rtimes D_{4}$ (shared)%
}%
}%
\hfill\copy\matricesbox
}%
\ifdim\wd\onelinebox>\myboxwidth
\hbox to \myboxwidth{%
$\Pi_{6,11}$ spans $L_{4.12}$%
\hfil
$2|2\slashinfty2|2\slashinfty\rtimes D_{4}$ (shared)%
}%
\box\matricesbox
\else
\hbox to \myboxwidth{%
\unhbox\onelinebox
}%
\fi
\else
\hbox to \myboxwidth{%
$\Pi_{6,11}$ spans $L_{4.12}$%
\hfil}%
\hbox to \myboxwidth{%
$2|2\slashinfty2|2\slashinfty\rtimes D_{4}$ (shared)%
\hfil}%
\box\matricesbox
\fi
}%
\hfill\discretionary{}{}{}%
\setbox\matricesbox=\hbox{%
{$\left[\!\llap{\phantom{%
\begingroup \smaller\smaller\smaller\begin{tabular}{@{}c@{}}%
\phantom{0}\\\phantom{0}\\\phantom{0}
\end{tabular}\endgroup%
}}\right.$}%
\begingroup \smaller\smaller\smaller\begin{tabular}{@{}c@{}}%
-1/4\\\phantom{0}\\\phantom{0}
\end{tabular}\endgroup%
\kern3pt%
\begingroup \smaller\smaller\smaller\begin{tabular}{@{}c@{}}%
\phantom{0}\\1/2\\\phantom{0}
\end{tabular}\endgroup%
\kern3pt%
\begingroup \smaller\smaller\smaller\begin{tabular}{@{}c@{}}%
\phantom{0}\\\phantom{0}\\21/2
\end{tabular}\endgroup%
{$\left.\llap{\phantom{%
\begingroup \smaller\smaller\smaller\begin{tabular}{@{}c@{}}%
\phantom{0}\\\phantom{0}\\\phantom{0}
\end{tabular}\endgroup%
}}\!\right]$}%
{$\left[\!\llap{\phantom{%
\begingroup \smaller\smaller\smaller\begin{tabular}{@{}c@{}}%
0\\0\\0
\end{tabular}\endgroup%
}}\right.$}%
\begingroup \smaller\smaller\smaller\begin{tabular}{@{}c@{}}%
4\\-4\\0
\end{tabular}\endgroup%
\kern3pt%
\begingroup \smaller\smaller\smaller\begin{tabular}{@{}c@{}}%
6\\-3\\-1
\end{tabular}\endgroup%
{$\left.\llap{\phantom{%
\begingroup \smaller\smaller\smaller\begin{tabular}{@{}c@{}}%
0\\0\\0
\end{tabular}\endgroup%
}}\!\right]$}%
}%
\ifdim\wd\matricesbox>\halfwidth\myboxwidth=\hsize\else\myboxwidth=\halfwidth\fi
\vbox{%
\ifdim\myboxwidth=\hsize
\setbox\onelinebox=\hbox{%
\vbox{\hbox{%
$\Pi_{6,12}$ spans $L_{22.1}$%
}\hbox{%
$2|2\slashthree2|2\slashthree\rtimes D_{4}$%
}%
}%
\hfill\copy\matricesbox
}%
\ifdim\wd\onelinebox>\myboxwidth
\hbox to \myboxwidth{%
$\Pi_{6,12}$ spans $L_{22.1}$%
\hfil
$2|2\slashthree2|2\slashthree\rtimes D_{4}$%
}%
\box\matricesbox
\else
\hbox to \myboxwidth{%
\unhbox\onelinebox
}%
\fi
\else
\hbox to \myboxwidth{%
$\Pi_{6,12}$ spans $L_{22.1}$%
\hfil}%
\hbox to \myboxwidth{%
$2|2\slashthree2|2\slashthree\rtimes D_{4}$%
\hfil}%
\box\matricesbox
\fi
}%
\hfill\discretionary{}{}{}%
\setbox\matricesbox=\hbox{%
{$\left[\!\llap{\phantom{%
\begingroup \smaller\smaller\smaller\begin{tabular}{@{}c@{}}%
\phantom{0}\\\phantom{0}\\\phantom{0}
\end{tabular}\endgroup%
}}\right.$}%
\begingroup \smaller\smaller\smaller\begin{tabular}{@{}c@{}}%
-5/8\\\phantom{0}\\\phantom{0}
\end{tabular}\endgroup%
\kern3pt%
\begingroup \smaller\smaller\smaller\begin{tabular}{@{}c@{}}%
\phantom{0}\\24\\\phantom{0}
\end{tabular}\endgroup%
\kern3pt%
\begingroup \smaller\smaller\smaller\begin{tabular}{@{}c@{}}%
\phantom{0}\\\phantom{0}\\9/2
\end{tabular}\endgroup%
{$\left.\llap{\phantom{%
\begingroup \smaller\smaller\smaller\begin{tabular}{@{}c@{}}%
\phantom{0}\\\phantom{0}\\\phantom{0}
\end{tabular}\endgroup%
}}\!\right]$}%
{$\left[\!\llap{\phantom{%
\begingroup \smaller\smaller\smaller\begin{tabular}{@{}c@{}}%
0\\0\\0
\end{tabular}\endgroup%
}}\right.$}%
\begingroup \smaller\smaller\smaller\begin{tabular}{@{}c@{}}%
6\\1\\1
\end{tabular}\endgroup%
\kern3pt%
\begingroup \smaller\smaller\smaller\begin{tabular}{@{}c@{}}%
2\\0\\1
\end{tabular}\endgroup%
{$\left.\llap{\phantom{%
\begingroup \smaller\smaller\smaller\begin{tabular}{@{}c@{}}%
0\\0\\0
\end{tabular}\endgroup%
}}\!\right]$}%
}%
\ifdim\wd\matricesbox>\halfwidth\myboxwidth=\hsize\else\myboxwidth=\halfwidth\fi
\vbox{%
\ifdim\myboxwidth=\hsize
\setbox\onelinebox=\hbox{%
\vbox{\hbox{%
$\Pi_{6,13}$ spans $L_{251.1}$%
}\hbox{%
$\slashthree6|6\slashthree6|6\rtimes D_{4}$%
}%
}%
\hfill\copy\matricesbox
}%
\ifdim\wd\onelinebox>\myboxwidth
\hbox to \myboxwidth{%
$\Pi_{6,13}$ spans $L_{251.1}$%
\hfil
$\slashthree6|6\slashthree6|6\rtimes D_{4}$%
}%
\box\matricesbox
\else
\hbox to \myboxwidth{%
\unhbox\onelinebox
}%
\fi
\else
\hbox to \myboxwidth{%
$\Pi_{6,13}$ spans $L_{251.1}$%
\hfil}%
\hbox to \myboxwidth{%
$\slashthree6|6\slashthree6|6\rtimes D_{4}$%
\hfil}%
\box\matricesbox
\fi
}%
\hfill\discretionary{}{}{}%
\setbox\matricesbox=\hbox{%
{$\left[\!\llap{\phantom{%
\begingroup \smaller\smaller\smaller\begin{tabular}{@{}c@{}}%
\phantom{0}\\\phantom{0}\\\phantom{0}
\end{tabular}\endgroup%
}}\right.$}%
\begingroup \smaller\smaller\smaller\begin{tabular}{@{}c@{}}%
-7/8\\\phantom{0}\\\phantom{0}
\end{tabular}\endgroup%
\kern3pt%
\begingroup \smaller\smaller\smaller\begin{tabular}{@{}c@{}}%
\phantom{0}\\4\\\phantom{0}
\end{tabular}\endgroup%
\kern3pt%
\begingroup \smaller\smaller\smaller\begin{tabular}{@{}c@{}}%
\phantom{0}\\\phantom{0}\\3/2
\end{tabular}\endgroup%
{$\left.\llap{\phantom{%
\begingroup \smaller\smaller\smaller\begin{tabular}{@{}c@{}}%
\phantom{0}\\\phantom{0}\\\phantom{0}
\end{tabular}\endgroup%
}}\!\right]$}%
{$\left[\!\llap{\phantom{%
\begingroup \smaller\smaller\smaller\begin{tabular}{@{}c@{}}%
0\\0\\0
\end{tabular}\endgroup%
}}\right.$}%
\begingroup \smaller\smaller\smaller\begin{tabular}{@{}c@{}}%
2\\-1\\-1
\end{tabular}\endgroup%
\kern3pt%
\begingroup \smaller\smaller\smaller\begin{tabular}{@{}c@{}}%
6\\0\\-5
\end{tabular}\endgroup%
{$\left.\llap{\phantom{%
\begingroup \smaller\smaller\smaller\begin{tabular}{@{}c@{}}%
0\\0\\0
\end{tabular}\endgroup%
}}\!\right]$}%
}%
\ifdim\wd\matricesbox>\halfwidth\myboxwidth=\hsize\else\myboxwidth=\halfwidth\fi
\vbox{%
\ifdim\myboxwidth=\hsize
\setbox\onelinebox=\hbox{%
\vbox{\hbox{%
$\Pi_{6,14}$ spans $L_{22.1}$%
}\hbox{%
$\slashthree6|6\slashthree6|6\rtimes D_{4}$%
}%
}%
\hfill\copy\matricesbox
}%
\ifdim\wd\onelinebox>\myboxwidth
\hbox to \myboxwidth{%
$\Pi_{6,14}$ spans $L_{22.1}$%
\hfil
$\slashthree6|6\slashthree6|6\rtimes D_{4}$%
}%
\box\matricesbox
\else
\hbox to \myboxwidth{%
\unhbox\onelinebox
}%
\fi
\else
\hbox to \myboxwidth{%
$\Pi_{6,14}$ spans $L_{22.1}$%
\hfil}%
\hbox to \myboxwidth{%
$\slashthree6|6\slashthree6|6\rtimes D_{4}$%
\hfil}%
\box\matricesbox
\fi
}%
\hfill\discretionary{}{}{}%
\setbox\matricesbox=\hbox{%
{$\left[\!\llap{\phantom{%
\begingroup \smaller\smaller\smaller\begin{tabular}{@{}c@{}}%
\phantom{0}\\\phantom{0}\\\phantom{0}
\end{tabular}\endgroup%
}}\right.$}%
\begingroup \smaller\smaller\smaller\begin{tabular}{@{}c@{}}%
-3/2\\\phantom{0}\\\phantom{0}
\end{tabular}\endgroup%
\kern3pt%
\begingroup \smaller\smaller\smaller\begin{tabular}{@{}c@{}}%
\phantom{0}\\1/2\\\phantom{0}
\end{tabular}\endgroup%
\kern3pt%
\begingroup \smaller\smaller\smaller\begin{tabular}{@{}c@{}}%
\phantom{0}\\\phantom{0}\\2
\end{tabular}\endgroup%
{$\left.\llap{\phantom{%
\begingroup \smaller\smaller\smaller\begin{tabular}{@{}c@{}}%
\phantom{0}\\\phantom{0}\\\phantom{0}
\end{tabular}\endgroup%
}}\!\right]$}%
{$\left[\!\llap{\phantom{%
\begingroup \smaller\smaller\smaller\begin{tabular}{@{}c@{}}%
0\\0\\0
\end{tabular}\endgroup%
}}\right.$}%
\begingroup \smaller\smaller\smaller\begin{tabular}{@{}c@{}}%
2\\4\\0
\end{tabular}\endgroup%
\kern3pt%
\begingroup \smaller\smaller\smaller\begin{tabular}{@{}c@{}}%
1\\1\\-1
\end{tabular}\endgroup%
{$\left.\llap{\phantom{%
\begingroup \smaller\smaller\smaller\begin{tabular}{@{}c@{}}%
0\\0\\0
\end{tabular}\endgroup%
}}\!\right]$}%
}%
\ifdim\wd\matricesbox>\halfwidth\myboxwidth=\hsize\else\myboxwidth=\halfwidth\fi
\vbox{%
\ifdim\myboxwidth=\hsize
\setbox\onelinebox=\hbox{%
\vbox{\hbox{%
$\Pi_{6,15}$ spans $L_{123.1}$%
}\hbox{%
$4|4\slashtwo4|4\slashtwo\rtimes D_{4}$%
}%
}%
\hfill\copy\matricesbox
}%
\ifdim\wd\onelinebox>\myboxwidth
\hbox to \myboxwidth{%
$\Pi_{6,15}$ spans $L_{123.1}$%
\hfil
$4|4\slashtwo4|4\slashtwo\rtimes D_{4}$%
}%
\box\matricesbox
\else
\hbox to \myboxwidth{%
\unhbox\onelinebox
}%
\fi
\else
\hbox to \myboxwidth{%
$\Pi_{6,15}$ spans $L_{123.1}$%
\hfil}%
\hbox to \myboxwidth{%
$4|4\slashtwo4|4\slashtwo\rtimes D_{4}$%
\hfil}%
\box\matricesbox
\fi
}%
\hfill\discretionary{}{}{}%
\setbox\matricesbox=\hbox{%
{$\left[\!\llap{\phantom{%
\begingroup \smaller\smaller\smaller\begin{tabular}{@{}c@{}}%
\phantom{0}\\\phantom{0}\\\phantom{0}
\end{tabular}\endgroup%
}}\right.$}%
\begingroup \smaller\smaller\smaller\begin{tabular}{@{}c@{}}%
-2\\\phantom{0}\\\phantom{0}
\end{tabular}\endgroup%
\kern3pt%
\begingroup \smaller\smaller\smaller\begin{tabular}{@{}c@{}}%
\phantom{0}\\1\\\phantom{0}
\end{tabular}\endgroup%
\kern3pt%
\begingroup \smaller\smaller\smaller\begin{tabular}{@{}c@{}}%
\phantom{0}\\\phantom{0}\\2
\end{tabular}\endgroup%
{$\left.\llap{\phantom{%
\begingroup \smaller\smaller\smaller\begin{tabular}{@{}c@{}}%
\phantom{0}\\\phantom{0}\\\phantom{0}
\end{tabular}\endgroup%
}}\!\right]$}%
{$\left[\!\llap{\phantom{%
\begingroup \smaller\smaller\smaller\begin{tabular}{@{}c@{}}%
0\\0\\0
\end{tabular}\endgroup%
}}\right.$}%
\begingroup \smaller\smaller\smaller\begin{tabular}{@{}c@{}}%
4\\6\\0
\end{tabular}\endgroup%
\kern3pt%
\begingroup \smaller\smaller\smaller\begin{tabular}{@{}c@{}}%
1\\1\\-1
\end{tabular}\endgroup%
{$\left.\llap{\phantom{%
\begingroup \smaller\smaller\smaller\begin{tabular}{@{}c@{}}%
0\\0\\0
\end{tabular}\endgroup%
}}\!\right]$}%
}%
\ifdim\wd\matricesbox>\halfwidth\myboxwidth=\hsize\else\myboxwidth=\halfwidth\fi
\vbox{%
\ifdim\myboxwidth=\hsize
\setbox\onelinebox=\hbox{%
\vbox{\hbox{%
$\Pi_{6,16}=\hbox{GN}_{44}$ spans $L_{141.14}$%
}\hbox{%
$\infty|\infty\slashinfty\infty|\infty\slashinfty\rtimes D_{4}$%
}%
}%
\hfill\copy\matricesbox
}%
\ifdim\wd\onelinebox>\myboxwidth
\hbox to \myboxwidth{%
$\Pi_{6,16}=\hbox{GN}_{44}$ spans $L_{141.14}$%
\hfil
$\infty|\infty\slashinfty\infty|\infty\slashinfty\rtimes D_{4}$%
}%
\box\matricesbox
\else
\hbox to \myboxwidth{%
\unhbox\onelinebox
}%
\fi
\else
\hbox to \myboxwidth{%
$\Pi_{6,16}=\hbox{GN}_{44}$ spans $L_{141.14}$%
\hfil}%
\hbox to \myboxwidth{%
$\infty|\infty\slashinfty\infty|\infty\slashinfty\rtimes D_{4}$%
\hfil}%
\box\matricesbox
\fi
}%
\hfill\discretionary{}{}{}%
\setbox\matricesbox=\hbox{%
{$\left[\!\llap{\phantom{%
\begingroup \smaller\smaller\smaller\begin{tabular}{@{}c@{}}%
\phantom{0}\\\phantom{0}\\\phantom{0}
\end{tabular}\endgroup%
}}\right.$}%
\begingroup \smaller\smaller\smaller\begin{tabular}{@{}c@{}}%
-5/8\\\phantom{0}\\\phantom{0}
\end{tabular}\endgroup%
\kern3pt%
\begingroup \smaller\smaller\smaller\begin{tabular}{@{}c@{}}%
\phantom{0}\\1/2\\\phantom{0}
\end{tabular}\endgroup%
\kern3pt%
\begingroup \smaller\smaller\smaller\begin{tabular}{@{}c@{}}%
\phantom{0}\\\phantom{0}\\40
\end{tabular}\endgroup%
{$\left.\llap{\phantom{%
\begingroup \smaller\smaller\smaller\begin{tabular}{@{}c@{}}%
\phantom{0}\\\phantom{0}\\\phantom{0}
\end{tabular}\endgroup%
}}\!\right]$}%
{$\left[\!\llap{\phantom{%
\begingroup \smaller\smaller\smaller\begin{tabular}{@{}c@{}}%
0\\0\\0
\end{tabular}\endgroup%
}}\right.$}%
\begingroup \smaller\smaller\smaller\begin{tabular}{@{}c@{}}%
2\\3\\0
\end{tabular}\endgroup%
\kern3pt%
\begingroup \smaller\smaller\smaller\begin{tabular}{@{}c@{}}%
8\\4\\-1
\end{tabular}\endgroup%
{$\left.\llap{\phantom{%
\begingroup \smaller\smaller\smaller\begin{tabular}{@{}c@{}}%
0\\0\\0
\end{tabular}\endgroup%
}}\!\right]$}%
}%
\ifdim\wd\matricesbox>\halfwidth\myboxwidth=\hsize\else\myboxwidth=\halfwidth\fi
\vbox{%
\ifdim\myboxwidth=\hsize
\setbox\onelinebox=\hbox{%
\vbox{\hbox{%
$\Pi_{6,17}=\hbox{GN}_{41}$ spans $L_{10.8}$%
}\hbox{%
$|\infty\slashinfty\infty|\infty\slashinfty\infty\rtimes D_{4}$%
}%
}%
\hfill\copy\matricesbox
}%
\ifdim\wd\onelinebox>\myboxwidth
\hbox to \myboxwidth{%
$\Pi_{6,17}=\hbox{GN}_{41}$ spans $L_{10.8}$%
\hfil
$|\infty\slashinfty\infty|\infty\slashinfty\infty\rtimes D_{4}$%
}%
\box\matricesbox
\else
\hbox to \myboxwidth{%
\unhbox\onelinebox
}%
\fi
\else
\hbox to \myboxwidth{%
$\Pi_{6,17}=\hbox{GN}_{41}$ spans $L_{10.8}$%
\hfil}%
\hbox to \myboxwidth{%
$|\infty\slashinfty\infty|\infty\slashinfty\infty\rtimes D_{4}$%
\hfil}%
\box\matricesbox
\fi
}%
\hfill\discretionary{}{}{}%
\setbox\matricesbox=\hbox{%
{$\left[\!\llap{\phantom{%
\begingroup \smaller\smaller\smaller\begin{tabular}{@{}c@{}}%
\phantom{0}\\\phantom{0}\\\phantom{0}
\end{tabular}\endgroup%
}}\right.$}%
\begingroup \smaller\smaller\smaller\begin{tabular}{@{}c@{}}%
-5/8\\\phantom{0}\\\phantom{0}
\end{tabular}\endgroup%
\kern3pt%
\begingroup \smaller\smaller\smaller\begin{tabular}{@{}c@{}}%
\phantom{0}\\1/2\\\phantom{0}
\end{tabular}\endgroup%
\kern3pt%
\begingroup \smaller\smaller\smaller\begin{tabular}{@{}c@{}}%
\phantom{0}\\\phantom{0}\\12
\end{tabular}\endgroup%
{$\left.\llap{\phantom{%
\begingroup \smaller\smaller\smaller\begin{tabular}{@{}c@{}}%
\phantom{0}\\\phantom{0}\\\phantom{0}
\end{tabular}\endgroup%
}}\!\right]$}%
{$\left[\!\llap{\phantom{%
\begingroup \smaller\smaller\smaller\begin{tabular}{@{}c@{}}%
0\\0\\0
\end{tabular}\endgroup%
}}\right.$}%
\begingroup \smaller\smaller\smaller\begin{tabular}{@{}c@{}}%
2\\-3\\0
\end{tabular}\endgroup%
\kern3pt%
\begingroup \smaller\smaller\smaller\begin{tabular}{@{}c@{}}%
4\\-2\\-1
\end{tabular}\endgroup%
{$\left.\llap{\phantom{%
\begingroup \smaller\smaller\smaller\begin{tabular}{@{}c@{}}%
0\\0\\0
\end{tabular}\endgroup%
}}\!\right]$}%
}%
\ifdim\wd\matricesbox>\halfwidth\myboxwidth=\hsize\else\myboxwidth=\halfwidth\fi
\vbox{%
\ifdim\myboxwidth=\hsize
\setbox\onelinebox=\hbox{%
\vbox{\hbox{%
$\Pi_{6,18}$ spans $L_{19.1}$%
}\hbox{%
$4|4\slashtwo4|4\slashtwo\rtimes D_{4}$%
}%
}%
\hfill\copy\matricesbox
}%
\ifdim\wd\onelinebox>\myboxwidth
\hbox to \myboxwidth{%
$\Pi_{6,18}$ spans $L_{19.1}$%
\hfil
$4|4\slashtwo4|4\slashtwo\rtimes D_{4}$%
}%
\box\matricesbox
\else
\hbox to \myboxwidth{%
\unhbox\onelinebox
}%
\fi
\else
\hbox to \myboxwidth{%
$\Pi_{6,18}$ spans $L_{19.1}$%
\hfil}%
\hbox to \myboxwidth{%
$4|4\slashtwo4|4\slashtwo\rtimes D_{4}$%
\hfil}%
\box\matricesbox
\fi
}%
\hfill\discretionary{}{}{}%
\setbox\matricesbox=\hbox{%
{$\left[\!\llap{\phantom{%
\begingroup \smaller\smaller\smaller\begin{tabular}{@{}c@{}}%
\phantom{0}\\\phantom{0}\\\phantom{0}
\end{tabular}\endgroup%
}}\right.$}%
\begingroup \smaller\smaller\smaller\begin{tabular}{@{}c@{}}%
-1/8\\\phantom{0}\\\phantom{0}
\end{tabular}\endgroup%
\kern3pt%
\begingroup \smaller\smaller\smaller\begin{tabular}{@{}c@{}}%
\phantom{0}\\28\\\phantom{0}
\end{tabular}\endgroup%
\kern3pt%
\begingroup \smaller\smaller\smaller\begin{tabular}{@{}c@{}}%
\phantom{0}\\\phantom{0}\\21/2
\end{tabular}\endgroup%
{$\left.\llap{\phantom{%
\begingroup \smaller\smaller\smaller\begin{tabular}{@{}c@{}}%
\phantom{0}\\\phantom{0}\\\phantom{0}
\end{tabular}\endgroup%
}}\!\right]$}%
{$\left[\!\llap{\phantom{%
\begingroup \smaller\smaller\smaller\begin{tabular}{@{}c@{}}%
0\\0\\0
\end{tabular}\endgroup%
}}\right.$}%
\begingroup \smaller\smaller\smaller\begin{tabular}{@{}c@{}}%
14\\1\\1
\end{tabular}\endgroup%
\kern3pt%
\begingroup \smaller\smaller\smaller\begin{tabular}{@{}c@{}}%
6\\0\\1
\end{tabular}\endgroup%
{$\left.\llap{\phantom{%
\begingroup \smaller\smaller\smaller\begin{tabular}{@{}c@{}}%
0\\0\\0
\end{tabular}\endgroup%
}}\!\right]$}%
}%
\ifdim\wd\matricesbox>\halfwidth\myboxwidth=\hsize\else\myboxwidth=\halfwidth\fi
\vbox{%
\ifdim\myboxwidth=\hsize
\setbox\onelinebox=\hbox{%
\vbox{\hbox{%
$\Pi_{6,19}$ spans $L_{22.8}$%
}\hbox{%
$\slashthree2|2\slashthree2|2\rtimes D_{4}$%
}%
}%
\hfill\copy\matricesbox
}%
\ifdim\wd\onelinebox>\myboxwidth
\hbox to \myboxwidth{%
$\Pi_{6,19}$ spans $L_{22.8}$%
\hfil
$\slashthree2|2\slashthree2|2\rtimes D_{4}$%
}%
\box\matricesbox
\else
\hbox to \myboxwidth{%
\unhbox\onelinebox
}%
\fi
\else
\hbox to \myboxwidth{%
$\Pi_{6,19}$ spans $L_{22.8}$%
\hfil}%
\hbox to \myboxwidth{%
$\slashthree2|2\slashthree2|2\rtimes D_{4}$%
\hfil}%
\box\matricesbox
\fi
}%
\hfill\discretionary{}{}{}%
\setbox\matricesbox=\hbox{%
{$\left[\!\llap{\phantom{%
\begingroup \smaller\smaller\smaller\begin{tabular}{@{}c@{}}%
\phantom{0}\\\phantom{0}\\\phantom{0}
\end{tabular}\endgroup%
}}\right.$}%
\begingroup \smaller\smaller\smaller\begin{tabular}{@{}c@{}}%
-1/8\\\phantom{0}\\\phantom{0}
\end{tabular}\endgroup%
\kern3pt%
\begingroup \smaller\smaller\smaller\begin{tabular}{@{}c@{}}%
\phantom{0}\\45\\\phantom{0}
\end{tabular}\endgroup%
\kern3pt%
\begingroup \smaller\smaller\smaller\begin{tabular}{@{}c@{}}%
\phantom{0}\\\phantom{0}\\3/2
\end{tabular}\endgroup%
{$\left.\llap{\phantom{%
\begingroup \smaller\smaller\smaller\begin{tabular}{@{}c@{}}%
\phantom{0}\\\phantom{0}\\\phantom{0}
\end{tabular}\endgroup%
}}\!\right]$}%
{$\left[\!\llap{\phantom{%
\begingroup \smaller\smaller\smaller\begin{tabular}{@{}c@{}}%
0\\0\\0
\end{tabular}\endgroup%
}}\right.$}%
\begingroup \smaller\smaller\smaller\begin{tabular}{@{}c@{}}%
18\\-1\\3
\end{tabular}\endgroup%
\kern3pt%
\begingroup \smaller\smaller\smaller\begin{tabular}{@{}c@{}}%
4\\0\\2
\end{tabular}\endgroup%
{$\left.\llap{\phantom{%
\begingroup \smaller\smaller\smaller\begin{tabular}{@{}c@{}}%
0\\0\\0
\end{tabular}\endgroup%
}}\!\right]$}%
}%
\ifdim\wd\matricesbox>\halfwidth\myboxwidth=\hsize\else\myboxwidth=\halfwidth\fi
\vbox{%
\ifdim\myboxwidth=\hsize
\setbox\onelinebox=\hbox{%
\vbox{\hbox{%
$\Pi_{6,20}$ spans $L_{16.12}$%
}\hbox{%
$\slashthree2|2\slashthree2|2\rtimes D_{4}$%
}%
}%
\hfill\copy\matricesbox
}%
\ifdim\wd\onelinebox>\myboxwidth
\hbox to \myboxwidth{%
$\Pi_{6,20}$ spans $L_{16.12}$%
\hfil
$\slashthree2|2\slashthree2|2\rtimes D_{4}$%
}%
\box\matricesbox
\else
\hbox to \myboxwidth{%
\unhbox\onelinebox
}%
\fi
\else
\hbox to \myboxwidth{%
$\Pi_{6,20}$ spans $L_{16.12}$%
\hfil}%
\hbox to \myboxwidth{%
$\slashthree2|2\slashthree2|2\rtimes D_{4}$%
\hfil}%
\box\matricesbox
\fi
}%
\hfill\discretionary{}{}{}%
\setbox\matricesbox=\hbox{%
{$\left[\!\llap{\phantom{%
\begingroup \smaller\smaller\smaller\begin{tabular}{@{}c@{}}%
\phantom{0}\\\phantom{0}\\\phantom{0}
\end{tabular}\endgroup%
}}\right.$}%
\begingroup \smaller\smaller\smaller\begin{tabular}{@{}c@{}}%
-1/8\\\phantom{0}\\\phantom{0}
\end{tabular}\endgroup%
\kern3pt%
\begingroup \smaller\smaller\smaller\begin{tabular}{@{}c@{}}%
\phantom{0}\\120\\\phantom{0}
\end{tabular}\endgroup%
\kern3pt%
\begingroup \smaller\smaller\smaller\begin{tabular}{@{}c@{}}%
\phantom{0}\\\phantom{0}\\5/2
\end{tabular}\endgroup%
{$\left.\llap{\phantom{%
\begingroup \smaller\smaller\smaller\begin{tabular}{@{}c@{}}%
\phantom{0}\\\phantom{0}\\\phantom{0}
\end{tabular}\endgroup%
}}\!\right]$}%
{$\left[\!\llap{\phantom{%
\begingroup \smaller\smaller\smaller\begin{tabular}{@{}c@{}}%
0\\0\\0
\end{tabular}\endgroup%
}}\right.$}%
\begingroup \smaller\smaller\smaller\begin{tabular}{@{}c@{}}%
30\\-1\\-3
\end{tabular}\endgroup%
\kern3pt%
\begingroup \smaller\smaller\smaller\begin{tabular}{@{}c@{}}%
2\\0\\-1
\end{tabular}\endgroup%
{$\left.\llap{\phantom{%
\begingroup \smaller\smaller\smaller\begin{tabular}{@{}c@{}}%
0\\0\\0
\end{tabular}\endgroup%
}}\!\right]$}%
}%
\ifdim\wd\matricesbox>\halfwidth\myboxwidth=\hsize\else\myboxwidth=\halfwidth\fi
\vbox{%
\ifdim\myboxwidth=\hsize
\setbox\onelinebox=\hbox{%
\vbox{\hbox{%
$\Pi_{6,21}$ spans $L_{31.10}$%
}\hbox{%
$\slashthree2|2\slashthree2|2\rtimes D_{4}$%
}%
}%
\hfill\copy\matricesbox
}%
\ifdim\wd\onelinebox>\myboxwidth
\hbox to \myboxwidth{%
$\Pi_{6,21}$ spans $L_{31.10}$%
\hfil
$\slashthree2|2\slashthree2|2\rtimes D_{4}$%
}%
\box\matricesbox
\else
\hbox to \myboxwidth{%
\unhbox\onelinebox
}%
\fi
\else
\hbox to \myboxwidth{%
$\Pi_{6,21}$ spans $L_{31.10}$%
\hfil}%
\hbox to \myboxwidth{%
$\slashthree2|2\slashthree2|2\rtimes D_{4}$%
\hfil}%
\box\matricesbox
\fi
}%
\hfill\discretionary{}{}{}%
\setbox\matricesbox=\hbox{%
{$\left[\!\llap{\phantom{%
\begingroup \smaller\smaller\smaller\begin{tabular}{@{}c@{}}%
\phantom{0}\\\phantom{0}\\\phantom{0}
\end{tabular}\endgroup%
}}\right.$}%
\begingroup \smaller\smaller\smaller\begin{tabular}{@{}c@{}}%
-1\\\phantom{0}\\\phantom{0}
\end{tabular}\endgroup%
\kern3pt%
\begingroup \smaller\smaller\smaller\begin{tabular}{@{}c@{}}%
\phantom{0}\\2\\\phantom{0}
\end{tabular}\endgroup%
\kern3pt%
\begingroup \smaller\smaller\smaller\begin{tabular}{@{}c@{}}%
\phantom{0}\\\phantom{0}\\18
\end{tabular}\endgroup%
{$\left.\llap{\phantom{%
\begingroup \smaller\smaller\smaller\begin{tabular}{@{}c@{}}%
\phantom{0}\\\phantom{0}\\\phantom{0}
\end{tabular}\endgroup%
}}\!\right]$}%
{$\left[\!\llap{\phantom{%
\begingroup \smaller\smaller\smaller\begin{tabular}{@{}c@{}}%
0\\0\\0
\end{tabular}\endgroup%
}}\right.$}%
\begingroup \smaller\smaller\smaller\begin{tabular}{@{}c@{}}%
1\\-1\\0
\end{tabular}\endgroup%
\kern3pt%
\begingroup \smaller\smaller\smaller\begin{tabular}{@{}c@{}}%
4\\-1\\-1
\end{tabular}\endgroup%
{$\left.\llap{\phantom{%
\begingroup \smaller\smaller\smaller\begin{tabular}{@{}c@{}}%
0\\0\\0
\end{tabular}\endgroup%
}}\!\right]$}%
}%
\ifdim\wd\matricesbox>\halfwidth\myboxwidth=\hsize\else\myboxwidth=\halfwidth\fi
\vbox{%
\ifdim\myboxwidth=\hsize
\setbox\onelinebox=\hbox{%
\vbox{\hbox{%
$\Pi_{6,22}=\hbox{GN}_{39}$ spans $L_{142.10}$%
}\hbox{%
$|\infty\slashtwo\infty|\infty\slashtwo\infty\rtimes D_{4}$%
}%
}%
\hfill\copy\matricesbox
}%
\ifdim\wd\onelinebox>\myboxwidth
\hbox to \myboxwidth{%
$\Pi_{6,22}=\hbox{GN}_{39}$ spans $L_{142.10}$%
\hfil
$|\infty\slashtwo\infty|\infty\slashtwo\infty\rtimes D_{4}$%
}%
\box\matricesbox
\else
\hbox to \myboxwidth{%
\unhbox\onelinebox
}%
\fi
\else
\hbox to \myboxwidth{%
$\Pi_{6,22}=\hbox{GN}_{39}$ spans $L_{142.10}$%
\hfil}%
\hbox to \myboxwidth{%
$|\infty\slashtwo\infty|\infty\slashtwo\infty\rtimes D_{4}$%
\hfil}%
\box\matricesbox
\fi
}%
\hfill\discretionary{}{}{}%
\setbox\matricesbox=\hbox{%
{$\left[\!\llap{\phantom{%
\begingroup \smaller\smaller\smaller\begin{tabular}{@{}c@{}}%
\phantom{0}\\\phantom{0}\\\phantom{0}
\end{tabular}\endgroup%
}}\right.$}%
\begingroup \smaller\smaller\smaller\begin{tabular}{@{}c@{}}%
-1\\\phantom{0}\\\phantom{0}
\end{tabular}\endgroup%
\kern3pt%
\begingroup \smaller\smaller\smaller\begin{tabular}{@{}c@{}}%
\phantom{0}\\4\\\phantom{0}
\end{tabular}\endgroup%
\kern3pt%
\begingroup \smaller\smaller\smaller\begin{tabular}{@{}c@{}}%
\phantom{0}\\\phantom{0}\\2
\end{tabular}\endgroup%
{$\left.\llap{\phantom{%
\begingroup \smaller\smaller\smaller\begin{tabular}{@{}c@{}}%
\phantom{0}\\\phantom{0}\\\phantom{0}
\end{tabular}\endgroup%
}}\!\right]$}%
{$\left[\!\llap{\phantom{%
\begingroup \smaller\smaller\smaller\begin{tabular}{@{}c@{}}%
0\\0\\0
\end{tabular}\endgroup%
}}\right.$}%
\begingroup \smaller\smaller\smaller\begin{tabular}{@{}c@{}}%
2\\1\\1
\end{tabular}\endgroup%
\kern3pt%
\begingroup \smaller\smaller\smaller\begin{tabular}{@{}c@{}}%
1\\0\\1
\end{tabular}\endgroup%
{$\left.\llap{\phantom{%
\begingroup \smaller\smaller\smaller\begin{tabular}{@{}c@{}}%
0\\0\\0
\end{tabular}\endgroup%
}}\!\right]$}%
}%
\ifdim\wd\matricesbox>\halfwidth\myboxwidth=\hsize\else\myboxwidth=\halfwidth\fi
\vbox{%
\ifdim\myboxwidth=\hsize
\setbox\onelinebox=\hbox{%
\vbox{\hbox{%
$\Pi_{6,23}$ spans $L_{141.3}$%
}\hbox{%
$\slashinfty2|2\slashinfty2|2\rtimes D_{4}$%
}%
}%
\hfill\copy\matricesbox
}%
\ifdim\wd\onelinebox>\myboxwidth
\hbox to \myboxwidth{%
$\Pi_{6,23}$ spans $L_{141.3}$%
\hfil
$\slashinfty2|2\slashinfty2|2\rtimes D_{4}$%
}%
\box\matricesbox
\else
\hbox to \myboxwidth{%
\unhbox\onelinebox
}%
\fi
\else
\hbox to \myboxwidth{%
$\Pi_{6,23}$ spans $L_{141.3}$%
\hfil}%
\hbox to \myboxwidth{%
$\slashinfty2|2\slashinfty2|2\rtimes D_{4}$%
\hfil}%
\box\matricesbox
\fi
}%
\hfill\discretionary{}{}{}%
\setbox\matricesbox=\hbox{%
{$\left[\!\llap{\phantom{%
\begingroup \smaller\smaller\smaller\begin{tabular}{@{}c@{}}%
\phantom{0}\\\phantom{0}\\\phantom{0}
\end{tabular}\endgroup%
}}\right.$}%
\begingroup \smaller\smaller\smaller\begin{tabular}{@{}c@{}}%
-1/2\\\phantom{0}\\\phantom{0}
\end{tabular}\endgroup%
\kern3pt%
\begingroup \smaller\smaller\smaller\begin{tabular}{@{}c@{}}%
\phantom{0}\\18\\\phantom{0}
\end{tabular}\endgroup%
\kern3pt%
\begingroup \smaller\smaller\smaller\begin{tabular}{@{}c@{}}%
\phantom{0}\\\phantom{0}\\3/2
\end{tabular}\endgroup%
{$\left.\llap{\phantom{%
\begingroup \smaller\smaller\smaller\begin{tabular}{@{}c@{}}%
\phantom{0}\\\phantom{0}\\\phantom{0}
\end{tabular}\endgroup%
}}\!\right]$}%
{$\left[\!\llap{\phantom{%
\begingroup \smaller\smaller\smaller\begin{tabular}{@{}c@{}}%
0\\0\\0
\end{tabular}\endgroup%
}}\right.$}%
\begingroup \smaller\smaller\smaller\begin{tabular}{@{}c@{}}%
6\\-1\\2
\end{tabular}\endgroup%
\kern3pt%
\begingroup \smaller\smaller\smaller\begin{tabular}{@{}c@{}}%
1\\0\\1
\end{tabular}\endgroup%
{$\left.\llap{\phantom{%
\begingroup \smaller\smaller\smaller\begin{tabular}{@{}c@{}}%
0\\0\\0
\end{tabular}\endgroup%
}}\!\right]$}%
}%
\ifdim\wd\matricesbox>\halfwidth\myboxwidth=\hsize\else\myboxwidth=\halfwidth\fi
\vbox{%
\ifdim\myboxwidth=\hsize
\setbox\onelinebox=\hbox{%
\vbox{\hbox{%
$\Pi_{6,24}$ spans $L_{4.22}$%
}\hbox{%
$\slashinfty2|2\slashinfty2|2\rtimes D_{4}$%
}%
}%
\hfill\copy\matricesbox
}%
\ifdim\wd\onelinebox>\myboxwidth
\hbox to \myboxwidth{%
$\Pi_{6,24}$ spans $L_{4.22}$%
\hfil
$\slashinfty2|2\slashinfty2|2\rtimes D_{4}$%
}%
\box\matricesbox
\else
\hbox to \myboxwidth{%
\unhbox\onelinebox
}%
\fi
\else
\hbox to \myboxwidth{%
$\Pi_{6,24}$ spans $L_{4.22}$%
\hfil}%
\hbox to \myboxwidth{%
$\slashinfty2|2\slashinfty2|2\rtimes D_{4}$%
\hfil}%
\box\matricesbox
\fi
}%
\hfill\discretionary{}{}{}%
\setbox\matricesbox=\hbox{%
{$\left[\!\llap{\phantom{%
\begingroup \smaller\smaller\smaller\begin{tabular}{@{}c@{}}%
\phantom{0}\\\phantom{0}\\\phantom{0}
\end{tabular}\endgroup%
}}\right.$}%
\begingroup \smaller\smaller\smaller\begin{tabular}{@{}c@{}}%
-1/8\\\phantom{0}\\\phantom{0}
\end{tabular}\endgroup%
\kern3pt%
\begingroup \smaller\smaller\smaller\begin{tabular}{@{}c@{}}%
\phantom{0}\\200\\\phantom{0}
\end{tabular}\endgroup%
\kern3pt%
\begingroup \smaller\smaller\smaller\begin{tabular}{@{}c@{}}%
\phantom{0}\\\phantom{0}\\5/2
\end{tabular}\endgroup%
{$\left.\llap{\phantom{%
\begingroup \smaller\smaller\smaller\begin{tabular}{@{}c@{}}%
\phantom{0}\\\phantom{0}\\\phantom{0}
\end{tabular}\endgroup%
}}\!\right]$}%
{$\left[\!\llap{\phantom{%
\begingroup \smaller\smaller\smaller\begin{tabular}{@{}c@{}}%
0\\0\\0
\end{tabular}\endgroup%
}}\right.$}%
\begingroup \smaller\smaller\smaller\begin{tabular}{@{}c@{}}%
40\\-1\\-4
\end{tabular}\endgroup%
\kern3pt%
\begingroup \smaller\smaller\smaller\begin{tabular}{@{}c@{}}%
2\\0\\-1
\end{tabular}\endgroup%
{$\left.\llap{\phantom{%
\begingroup \smaller\smaller\smaller\begin{tabular}{@{}c@{}}%
0\\0\\0
\end{tabular}\endgroup%
}}\!\right]$}%
}%
\ifdim\wd\matricesbox>\halfwidth\myboxwidth=\hsize\else\myboxwidth=\halfwidth\fi
\vbox{%
\ifdim\myboxwidth=\hsize
\setbox\onelinebox=\hbox{%
\vbox{\hbox{%
$\Pi_{6,25}$ spans $L_{10.13}$%
}\hbox{%
$\slashinfty2|2\slashinfty2|2\rtimes D_{4}$%
}%
}%
\hfill\copy\matricesbox
}%
\ifdim\wd\onelinebox>\myboxwidth
\hbox to \myboxwidth{%
$\Pi_{6,25}$ spans $L_{10.13}$%
\hfil
$\slashinfty2|2\slashinfty2|2\rtimes D_{4}$%
}%
\box\matricesbox
\else
\hbox to \myboxwidth{%
\unhbox\onelinebox
}%
\fi
\else
\hbox to \myboxwidth{%
$\Pi_{6,25}$ spans $L_{10.13}$%
\hfil}%
\hbox to \myboxwidth{%
$\slashinfty2|2\slashinfty2|2\rtimes D_{4}$%
\hfil}%
\box\matricesbox
\fi
}%
\hfill\discretionary{}{}{}%
\setbox\matricesbox=\hbox{%
{$\left[\!\llap{\phantom{%
\begingroup \smaller\smaller\smaller\begin{tabular}{@{}c@{}}%
\phantom{0}\\\phantom{0}\\\phantom{0}
\end{tabular}\endgroup%
}}\right.$}%
\begingroup \smaller\smaller\smaller\begin{tabular}{@{}c@{}}%
-1/5\\\phantom{0}\\\phantom{0}
\end{tabular}\endgroup%
\kern3pt%
\begingroup \smaller\smaller\smaller\begin{tabular}{@{}c@{}}%
\phantom{0}\\3/10\\\phantom{0}
\end{tabular}\endgroup%
\kern3pt%
\begingroup \smaller\smaller\smaller\begin{tabular}{@{}c@{}}%
\phantom{0}\\\phantom{0}\\21/2
\end{tabular}\endgroup%
{$\left.\llap{\phantom{%
\begingroup \smaller\smaller\smaller\begin{tabular}{@{}c@{}}%
\phantom{0}\\\phantom{0}\\\phantom{0}
\end{tabular}\endgroup%
}}\!\right]$}%
{$\left[\!\llap{\phantom{%
\begingroup \smaller\smaller\smaller\begin{tabular}{@{}c@{}}%
0\\0\\0
\end{tabular}\endgroup%
}}\right.$}%
\begingroup \smaller\smaller\smaller\begin{tabular}{@{}c@{}}%
1\\-2\\0
\end{tabular}\endgroup%
\kern3pt%
\begingroup \smaller\smaller\smaller\begin{tabular}{@{}c@{}}%
21\\-7\\-3
\end{tabular}\endgroup%
\kern3pt%
\begingroup \smaller\smaller\smaller\begin{tabular}{@{}c@{}}%
6\\3\\-1
\end{tabular}\endgroup%
\kern3pt%
\begingroup \smaller\smaller\smaller\begin{tabular}{@{}c@{}}%
3\\4\\0
\end{tabular}\endgroup%
{$\left.\llap{\phantom{%
\begingroup \smaller\smaller\smaller\begin{tabular}{@{}c@{}}%
0\\0\\0
\end{tabular}\endgroup%
}}\!\right]$}%
}%
\ifdim\wd\matricesbox>\halfwidth\myboxwidth=\hsize\else\myboxwidth=\halfwidth\fi
\vbox{%
\ifdim\myboxwidth=\hsize
\setbox\onelinebox=\hbox{%
\vbox{\hbox{%
$\Pi_{6,26}$ spans $L_{25.5}$%
}\hbox{%
$|222|222\rtimes D_{2}$%
}%
}%
\hfill\copy\matricesbox
}%
\ifdim\wd\onelinebox>\myboxwidth
\hbox to \myboxwidth{%
$\Pi_{6,26}$ spans $L_{25.5}$%
\hfil
$|222|222\rtimes D_{2}$%
}%
\box\matricesbox
\else
\hbox to \myboxwidth{%
\unhbox\onelinebox
}%
\fi
\else
\hbox to \myboxwidth{%
$\Pi_{6,26}$ spans $L_{25.5}$%
\hfil}%
\hbox to \myboxwidth{%
$|222|222\rtimes D_{2}$%
\hfil}%
\box\matricesbox
\fi
}%
\hfill\discretionary{}{}{}%
\setbox\matricesbox=\hbox{%
{$\left[\!\llap{\phantom{%
\begingroup \smaller\smaller\smaller\begin{tabular}{@{}c@{}}%
\phantom{0}\\\phantom{0}\\\phantom{0}
\end{tabular}\endgroup%
}}\right.$}%
\begingroup \smaller\smaller\smaller\begin{tabular}{@{}c@{}}%
-1/4\\\phantom{0}\\\phantom{0}
\end{tabular}\endgroup%
\kern3pt%
\begingroup \smaller\smaller\smaller\begin{tabular}{@{}c@{}}%
\phantom{0}\\3/4\\\phantom{0}
\end{tabular}\endgroup%
\kern3pt%
\begingroup \smaller\smaller\smaller\begin{tabular}{@{}c@{}}%
\phantom{0}\\\phantom{0}\\9/2
\end{tabular}\endgroup%
{$\left.\llap{\phantom{%
\begingroup \smaller\smaller\smaller\begin{tabular}{@{}c@{}}%
\phantom{0}\\\phantom{0}\\\phantom{0}
\end{tabular}\endgroup%
}}\!\right]$}%
{$\left[\!\llap{\phantom{%
\begingroup \smaller\smaller\smaller\begin{tabular}{@{}c@{}}%
0\\0\\0
\end{tabular}\endgroup%
}}\right.$}%
\begingroup \smaller\smaller\smaller\begin{tabular}{@{}c@{}}%
2\\2\\0
\end{tabular}\endgroup%
\kern3pt%
\begingroup \smaller\smaller\smaller\begin{tabular}{@{}c@{}}%
18\\6\\-4
\end{tabular}\endgroup%
\kern3pt%
\begingroup \smaller\smaller\smaller\begin{tabular}{@{}c@{}}%
3\\-1\\-1
\end{tabular}\endgroup%
\kern3pt%
\begingroup \smaller\smaller\smaller\begin{tabular}{@{}c@{}}%
2\\-2\\0
\end{tabular}\endgroup%
{$\left.\llap{\phantom{%
\begingroup \smaller\smaller\smaller\begin{tabular}{@{}c@{}}%
0\\0\\0
\end{tabular}\endgroup%
}}\!\right]$}%
}%
\ifdim\wd\matricesbox>\halfwidth\myboxwidth=\hsize\else\myboxwidth=\halfwidth\fi
\vbox{%
\ifdim\myboxwidth=\hsize
\setbox\onelinebox=\hbox{%
\vbox{\hbox{%
$\Pi_{6,27}$ spans $L_{172.4}$%
}\hbox{%
$|222|222\rtimes D_{2}$%
}%
}%
\hfill\copy\matricesbox
}%
\ifdim\wd\onelinebox>\myboxwidth
\hbox to \myboxwidth{%
$\Pi_{6,27}$ spans $L_{172.4}$%
\hfil
$|222|222\rtimes D_{2}$%
}%
\box\matricesbox
\else
\hbox to \myboxwidth{%
\unhbox\onelinebox
}%
\fi
\else
\hbox to \myboxwidth{%
$\Pi_{6,27}$ spans $L_{172.4}$%
\hfil}%
\hbox to \myboxwidth{%
$|222|222\rtimes D_{2}$%
\hfil}%
\box\matricesbox
\fi
}%
\hfill\discretionary{}{}{}%
\setbox\matricesbox=\hbox{%
{$\left[\!\llap{\phantom{%
\begingroup \smaller\smaller\smaller\begin{tabular}{@{}c@{}}%
\phantom{0}\\\phantom{0}\\\phantom{0}
\end{tabular}\endgroup%
}}\right.$}%
\begingroup \smaller\smaller\smaller\begin{tabular}{@{}c@{}}%
-1/4\\\phantom{0}\\\phantom{0}
\end{tabular}\endgroup%
\kern3pt%
\begingroup \smaller\smaller\smaller\begin{tabular}{@{}c@{}}%
\phantom{0}\\1/2\\\phantom{0}
\end{tabular}\endgroup%
\kern3pt%
\begingroup \smaller\smaller\smaller\begin{tabular}{@{}c@{}}%
\phantom{0}\\\phantom{0}\\5/2
\end{tabular}\endgroup%
{$\left.\llap{\phantom{%
\begingroup \smaller\smaller\smaller\begin{tabular}{@{}c@{}}%
\phantom{0}\\\phantom{0}\\\phantom{0}
\end{tabular}\endgroup%
}}\!\right]$}%
{$\left[\!\llap{\phantom{%
\begingroup \smaller\smaller\smaller\begin{tabular}{@{}c@{}}%
0\\0\\0
\end{tabular}\endgroup%
}}\right.$}%
\begingroup \smaller\smaller\smaller\begin{tabular}{@{}c@{}}%
4\\4\\0
\end{tabular}\endgroup%
\kern3pt%
\begingroup \smaller\smaller\smaller\begin{tabular}{@{}c@{}}%
10\\5\\3
\end{tabular}\endgroup%
\kern3pt%
\begingroup \smaller\smaller\smaller\begin{tabular}{@{}c@{}}%
2\\-1\\1
\end{tabular}\endgroup%
\kern3pt%
\begingroup \smaller\smaller\smaller\begin{tabular}{@{}c@{}}%
4\\-4\\0
\end{tabular}\endgroup%
{$\left.\llap{\phantom{%
\begingroup \smaller\smaller\smaller\begin{tabular}{@{}c@{}}%
0\\0\\0
\end{tabular}\endgroup%
}}\!\right]$}%
}%
\ifdim\wd\matricesbox>\halfwidth\myboxwidth=\hsize\else\myboxwidth=\halfwidth\fi
\vbox{%
\ifdim\myboxwidth=\hsize
\setbox\onelinebox=\hbox{%
\vbox{\hbox{%
$\Pi_{6,28}$ spans $L_{6.1}$%
}\hbox{%
$22|222|2\rtimes D_{2}$%
}%
}%
\hfill\copy\matricesbox
}%
\ifdim\wd\onelinebox>\myboxwidth
\hbox to \myboxwidth{%
$\Pi_{6,28}$ spans $L_{6.1}$%
\hfil
$22|222|2\rtimes D_{2}$%
}%
\box\matricesbox
\else
\hbox to \myboxwidth{%
\unhbox\onelinebox
}%
\fi
\else
\hbox to \myboxwidth{%
$\Pi_{6,28}$ spans $L_{6.1}$%
\hfil}%
\hbox to \myboxwidth{%
$22|222|2\rtimes D_{2}$%
\hfil}%
\box\matricesbox
\fi
}%
\hfill\discretionary{}{}{}%
\setbox\matricesbox=\hbox{%
{$\left[\!\llap{\phantom{%
\begingroup \smaller\smaller\smaller\begin{tabular}{@{}c@{}}%
\phantom{0}\\\phantom{0}\\\phantom{0}
\end{tabular}\endgroup%
}}\right.$}%
\begingroup \smaller\smaller\smaller\begin{tabular}{@{}c@{}}%
-1/14\\\phantom{0}\\\phantom{0}
\end{tabular}\endgroup%
\kern3pt%
\begingroup \smaller\smaller\smaller\begin{tabular}{@{}c@{}}%
\phantom{0}\\60/7\\\phantom{0}
\end{tabular}\endgroup%
\kern3pt%
\begingroup \smaller\smaller\smaller\begin{tabular}{@{}c@{}}%
\phantom{0}\\\phantom{0}\\5/2
\end{tabular}\endgroup%
{$\left.\llap{\phantom{%
\begingroup \smaller\smaller\smaller\begin{tabular}{@{}c@{}}%
\phantom{0}\\\phantom{0}\\\phantom{0}
\end{tabular}\endgroup%
}}\!\right]$}%
{$\left[\!\llap{\phantom{%
\begingroup \smaller\smaller\smaller\begin{tabular}{@{}c@{}}%
0\\0\\0
\end{tabular}\endgroup%
}}\right.$}%
\begingroup \smaller\smaller\smaller\begin{tabular}{@{}c@{}}%
6\\1\\0
\end{tabular}\endgroup%
\kern3pt%
\begingroup \smaller\smaller\smaller\begin{tabular}{@{}c@{}}%
20\\1\\4
\end{tabular}\endgroup%
\kern3pt%
\begingroup \smaller\smaller\smaller\begin{tabular}{@{}c@{}}%
15\\-1\\3
\end{tabular}\endgroup%
\kern3pt%
\begingroup \smaller\smaller\smaller\begin{tabular}{@{}c@{}}%
16\\-2\\0
\end{tabular}\endgroup%
{$\left.\llap{\phantom{%
\begingroup \smaller\smaller\smaller\begin{tabular}{@{}c@{}}%
0\\0\\0
\end{tabular}\endgroup%
}}\!\right]$}%
}%
\ifdim\wd\matricesbox>\halfwidth\myboxwidth=\hsize\else\myboxwidth=\halfwidth\fi
\vbox{%
\ifdim\myboxwidth=\hsize
\setbox\onelinebox=\hbox{%
\vbox{\hbox{%
$\Pi_{6,29}$ spans $L_{17.20}$%
}\hbox{%
$|222|222\rtimes D_{2}$%
}%
}%
\hfill\copy\matricesbox
}%
\ifdim\wd\onelinebox>\myboxwidth
\hbox to \myboxwidth{%
$\Pi_{6,29}$ spans $L_{17.20}$%
\hfil
$|222|222\rtimes D_{2}$%
}%
\box\matricesbox
\else
\hbox to \myboxwidth{%
\unhbox\onelinebox
}%
\fi
\else
\hbox to \myboxwidth{%
$\Pi_{6,29}$ spans $L_{17.20}$%
\hfil}%
\hbox to \myboxwidth{%
$|222|222\rtimes D_{2}$%
\hfil}%
\box\matricesbox
\fi
}%
\hfill\discretionary{}{}{}%
\setbox\matricesbox=\hbox{%
{$\left[\!\llap{\phantom{%
\begingroup \smaller\smaller\smaller\begin{tabular}{@{}c@{}}%
\phantom{0}\\\phantom{0}\\\phantom{0}
\end{tabular}\endgroup%
}}\right.$}%
\begingroup \smaller\smaller\smaller\begin{tabular}{@{}c@{}}%
-1/5\\\phantom{0}\\\phantom{0}
\end{tabular}\endgroup%
\kern3pt%
\begingroup \smaller\smaller\smaller\begin{tabular}{@{}c@{}}%
\phantom{0}\\6/5\\\phantom{0}
\end{tabular}\endgroup%
\kern3pt%
\begingroup \smaller\smaller\smaller\begin{tabular}{@{}c@{}}%
\phantom{0}\\\phantom{0}\\12
\end{tabular}\endgroup%
{$\left.\llap{\phantom{%
\begingroup \smaller\smaller\smaller\begin{tabular}{@{}c@{}}%
\phantom{0}\\\phantom{0}\\\phantom{0}
\end{tabular}\endgroup%
}}\!\right]$}%
{$\left[\!\llap{\phantom{%
\begingroup \smaller\smaller\smaller\begin{tabular}{@{}c@{}}%
0\\0\\0
\end{tabular}\endgroup%
}}\right.$}%
\begingroup \smaller\smaller\smaller\begin{tabular}{@{}c@{}}%
3\\-2\\0
\end{tabular}\endgroup%
\kern3pt%
\begingroup \smaller\smaller\smaller\begin{tabular}{@{}c@{}}%
16\\-4\\2
\end{tabular}\endgroup%
\kern3pt%
\begingroup \smaller\smaller\smaller\begin{tabular}{@{}c@{}}%
6\\1\\1
\end{tabular}\endgroup%
\kern3pt%
\begingroup \smaller\smaller\smaller\begin{tabular}{@{}c@{}}%
1\\1\\0
\end{tabular}\endgroup%
{$\left.\llap{\phantom{%
\begingroup \smaller\smaller\smaller\begin{tabular}{@{}c@{}}%
0\\0\\0
\end{tabular}\endgroup%
}}\!\right]$}%
}%
\ifdim\wd\matricesbox>\halfwidth\myboxwidth=\hsize\else\myboxwidth=\halfwidth\fi
\vbox{%
\ifdim\myboxwidth=\hsize
\setbox\onelinebox=\hbox{%
\vbox{\hbox{%
$\Pi_{6,30}$ spans $L_{151.10}$%
}\hbox{%
$22|222|2\rtimes D_{2}$ (shared)%
}%
}%
\hfill\copy\matricesbox
}%
\ifdim\wd\onelinebox>\myboxwidth
\hbox to \myboxwidth{%
$\Pi_{6,30}$ spans $L_{151.10}$%
\hfil
$22|222|2\rtimes D_{2}$ (shared)%
}%
\box\matricesbox
\else
\hbox to \myboxwidth{%
\unhbox\onelinebox
}%
\fi
\else
\hbox to \myboxwidth{%
$\Pi_{6,30}$ spans $L_{151.10}$%
\hfil}%
\hbox to \myboxwidth{%
$22|222|2\rtimes D_{2}$ (shared)%
\hfil}%
\box\matricesbox
\fi
}%
\hfill\discretionary{}{}{}%
\setbox\matricesbox=\hbox{%
{$\left[\!\llap{\phantom{%
\begingroup \smaller\smaller\smaller\begin{tabular}{@{}c@{}}%
\phantom{0}\\\phantom{0}\\\phantom{0}
\end{tabular}\endgroup%
}}\right.$}%
\begingroup \smaller\smaller\smaller\begin{tabular}{@{}c@{}}%
-1/8\\\phantom{0}\\\phantom{0}
\end{tabular}\endgroup%
\kern3pt%
\begingroup \smaller\smaller\smaller\begin{tabular}{@{}c@{}}%
\phantom{0}\\5/2\\\phantom{0}
\end{tabular}\endgroup%
\kern3pt%
\begingroup \smaller\smaller\smaller\begin{tabular}{@{}c@{}}%
\phantom{0}\\\phantom{0}\\20
\end{tabular}\endgroup%
{$\left.\llap{\phantom{%
\begingroup \smaller\smaller\smaller\begin{tabular}{@{}c@{}}%
\phantom{0}\\\phantom{0}\\\phantom{0}
\end{tabular}\endgroup%
}}\!\right]$}%
{$\left[\!\llap{\phantom{%
\begingroup \smaller\smaller\smaller\begin{tabular}{@{}c@{}}%
0\\0\\0
\end{tabular}\endgroup%
}}\right.$}%
\begingroup \smaller\smaller\smaller\begin{tabular}{@{}c@{}}%
2\\-1\\0
\end{tabular}\endgroup%
\kern3pt%
\begingroup \smaller\smaller\smaller\begin{tabular}{@{}c@{}}%
80\\-8\\6
\end{tabular}\endgroup%
\kern3pt%
\begingroup \smaller\smaller\smaller\begin{tabular}{@{}c@{}}%
10\\1\\1
\end{tabular}\endgroup%
\kern3pt%
\begingroup \smaller\smaller\smaller\begin{tabular}{@{}c@{}}%
2\\1\\0
\end{tabular}\endgroup%
{$\left.\llap{\phantom{%
\begingroup \smaller\smaller\smaller\begin{tabular}{@{}c@{}}%
0\\0\\0
\end{tabular}\endgroup%
}}\!\right]$}%
}%
\ifdim\wd\matricesbox>\halfwidth\myboxwidth=\hsize\else\myboxwidth=\halfwidth\fi
\vbox{%
\ifdim\myboxwidth=\hsize
\setbox\onelinebox=\hbox{%
\vbox{\hbox{%
$\Pi_{6,31}$ spans $L_{6.5}$%
}\hbox{%
$22|222|2\rtimes D_{2}$%
}%
}%
\hfill\copy\matricesbox
}%
\ifdim\wd\onelinebox>\myboxwidth
\hbox to \myboxwidth{%
$\Pi_{6,31}$ spans $L_{6.5}$%
\hfil
$22|222|2\rtimes D_{2}$%
}%
\box\matricesbox
\else
\hbox to \myboxwidth{%
\unhbox\onelinebox
}%
\fi
\else
\hbox to \myboxwidth{%
$\Pi_{6,31}$ spans $L_{6.5}$%
\hfil}%
\hbox to \myboxwidth{%
$22|222|2\rtimes D_{2}$%
\hfil}%
\box\matricesbox
\fi
}%
\hfill\discretionary{}{}{}%
\setbox\matricesbox=\hbox{%
{$\left[\!\llap{\phantom{%
\begingroup \smaller\smaller\smaller\begin{tabular}{@{}c@{}}%
\phantom{0}\\\phantom{0}\\\phantom{0}
\end{tabular}\endgroup%
}}\right.$}%
\begingroup \smaller\smaller\smaller\begin{tabular}{@{}c@{}}%
-1/8\\\phantom{0}\\\phantom{0}
\end{tabular}\endgroup%
\kern3pt%
\begingroup \smaller\smaller\smaller\begin{tabular}{@{}c@{}}%
\phantom{0}\\5/2\\\phantom{0}
\end{tabular}\endgroup%
\kern3pt%
\begingroup \smaller\smaller\smaller\begin{tabular}{@{}c@{}}%
\phantom{0}\\\phantom{0}\\20
\end{tabular}\endgroup%
{$\left.\llap{\phantom{%
\begingroup \smaller\smaller\smaller\begin{tabular}{@{}c@{}}%
\phantom{0}\\\phantom{0}\\\phantom{0}
\end{tabular}\endgroup%
}}\!\right]$}%
{$\left[\!\llap{\phantom{%
\begingroup \smaller\smaller\smaller\begin{tabular}{@{}c@{}}%
0\\0\\0
\end{tabular}\endgroup%
}}\right.$}%
\begingroup \smaller\smaller\smaller\begin{tabular}{@{}c@{}}%
2\\1\\0
\end{tabular}\endgroup%
\kern3pt%
\begingroup \smaller\smaller\smaller\begin{tabular}{@{}c@{}}%
80\\8\\6
\end{tabular}\endgroup%
\kern3pt%
\begingroup \smaller\smaller\smaller\begin{tabular}{@{}c@{}}%
10\\-1\\1
\end{tabular}\endgroup%
\kern3pt%
\begingroup \smaller\smaller\smaller\begin{tabular}{@{}c@{}}%
10\\-3\\0
\end{tabular}\endgroup%
{$\left.\llap{\phantom{%
\begingroup \smaller\smaller\smaller\begin{tabular}{@{}c@{}}%
0\\0\\0
\end{tabular}\endgroup%
}}\!\right]$}%
}%
\ifdim\wd\matricesbox>\halfwidth\myboxwidth=\hsize\else\myboxwidth=\halfwidth\fi
\vbox{%
\ifdim\myboxwidth=\hsize
\setbox\onelinebox=\hbox{%
\vbox{\hbox{%
$\Pi_{6,32}$ spans $L_{187.2}$%
}\hbox{%
$22|223|3\rtimes D_{2}$%
}%
}%
\hfill\copy\matricesbox
}%
\ifdim\wd\onelinebox>\myboxwidth
\hbox to \myboxwidth{%
$\Pi_{6,32}$ spans $L_{187.2}$%
\hfil
$22|223|3\rtimes D_{2}$%
}%
\box\matricesbox
\else
\hbox to \myboxwidth{%
\unhbox\onelinebox
}%
\fi
\else
\hbox to \myboxwidth{%
$\Pi_{6,32}$ spans $L_{187.2}$%
\hfil}%
\hbox to \myboxwidth{%
$22|223|3\rtimes D_{2}$%
\hfil}%
\box\matricesbox
\fi
}%
\hfill\discretionary{}{}{}%
\setbox\matricesbox=\hbox{%
{$\left[\!\llap{\phantom{%
\begingroup \smaller\smaller\smaller\begin{tabular}{@{}c@{}}%
\phantom{0}\\\phantom{0}\\\phantom{0}
\end{tabular}\endgroup%
}}\right.$}%
\begingroup \smaller\smaller\smaller\begin{tabular}{@{}c@{}}%
-1/2\\\phantom{0}\\\phantom{0}
\end{tabular}\endgroup%
\kern3pt%
\begingroup \smaller\smaller\smaller\begin{tabular}{@{}c@{}}%
\phantom{0}\\3/2\\\phantom{0}
\end{tabular}\endgroup%
\kern3pt%
\begingroup \smaller\smaller\smaller\begin{tabular}{@{}c@{}}%
\phantom{0}\\\phantom{0}\\6
\end{tabular}\endgroup%
{$\left.\llap{\phantom{%
\begingroup \smaller\smaller\smaller\begin{tabular}{@{}c@{}}%
\phantom{0}\\\phantom{0}\\\phantom{0}
\end{tabular}\endgroup%
}}\!\right]$}%
{$\left[\!\llap{\phantom{%
\begingroup \smaller\smaller\smaller\begin{tabular}{@{}c@{}}%
0\\0\\0
\end{tabular}\endgroup%
}}\right.$}%
\begingroup \smaller\smaller\smaller\begin{tabular}{@{}c@{}}%
1\\1\\0
\end{tabular}\endgroup%
\kern3pt%
\begingroup \smaller\smaller\smaller\begin{tabular}{@{}c@{}}%
3\\1\\1
\end{tabular}\endgroup%
\kern3pt%
\begingroup \smaller\smaller\smaller\begin{tabular}{@{}c@{}}%
3\\-1\\1
\end{tabular}\endgroup%
\kern3pt%
\begingroup \smaller\smaller\smaller\begin{tabular}{@{}c@{}}%
6\\-4\\0
\end{tabular}\endgroup%
{$\left.\llap{\phantom{%
\begingroup \smaller\smaller\smaller\begin{tabular}{@{}c@{}}%
0\\0\\0
\end{tabular}\endgroup%
}}\!\right]$}%
}%
\ifdim\wd\matricesbox>\halfwidth\myboxwidth=\hsize\else\myboxwidth=\halfwidth\fi
\vbox{%
\ifdim\myboxwidth=\hsize
\setbox\onelinebox=\hbox{%
\vbox{\hbox{%
$\Pi_{6,33}$ spans $L_{123.6}$%
}\hbox{%
$22|224|4\rtimes D_{2}$%
}%
}%
\hfill\copy\matricesbox
}%
\ifdim\wd\onelinebox>\myboxwidth
\hbox to \myboxwidth{%
$\Pi_{6,33}$ spans $L_{123.6}$%
\hfil
$22|224|4\rtimes D_{2}$%
}%
\box\matricesbox
\else
\hbox to \myboxwidth{%
\unhbox\onelinebox
}%
\fi
\else
\hbox to \myboxwidth{%
$\Pi_{6,33}$ spans $L_{123.6}$%
\hfil}%
\hbox to \myboxwidth{%
$22|224|4\rtimes D_{2}$%
\hfil}%
\box\matricesbox
\fi
}%
\hfill\discretionary{}{}{}%
\setbox\matricesbox=\hbox{%
{$\left[\!\llap{\phantom{%
\begingroup \smaller\smaller\smaller\begin{tabular}{@{}c@{}}%
\phantom{0}\\\phantom{0}\\\phantom{0}
\end{tabular}\endgroup%
}}\right.$}%
\begingroup \smaller\smaller\smaller\begin{tabular}{@{}c@{}}%
-1/5\\\phantom{0}\\\phantom{0}
\end{tabular}\endgroup%
\kern3pt%
\begingroup \smaller\smaller\smaller\begin{tabular}{@{}c@{}}%
\phantom{0}\\3/10\\\phantom{0}
\end{tabular}\endgroup%
\kern3pt%
\begingroup \smaller\smaller\smaller\begin{tabular}{@{}c@{}}%
\phantom{0}\\\phantom{0}\\45/2
\end{tabular}\endgroup%
{$\left.\llap{\phantom{%
\begingroup \smaller\smaller\smaller\begin{tabular}{@{}c@{}}%
\phantom{0}\\\phantom{0}\\\phantom{0}
\end{tabular}\endgroup%
}}\!\right]$}%
{$\left[\!\llap{\phantom{%
\begingroup \smaller\smaller\smaller\begin{tabular}{@{}c@{}}%
0\\0\\0
\end{tabular}\endgroup%
}}\right.$}%
\begingroup \smaller\smaller\smaller\begin{tabular}{@{}c@{}}%
3\\4\\0
\end{tabular}\endgroup%
\kern3pt%
\begingroup \smaller\smaller\smaller\begin{tabular}{@{}c@{}}%
10\\5\\1
\end{tabular}\endgroup%
\kern3pt%
\begingroup \smaller\smaller\smaller\begin{tabular}{@{}c@{}}%
9\\-3\\1
\end{tabular}\endgroup%
\kern3pt%
\begingroup \smaller\smaller\smaller\begin{tabular}{@{}c@{}}%
1\\-2\\0
\end{tabular}\endgroup%
{$\left.\llap{\phantom{%
\begingroup \smaller\smaller\smaller\begin{tabular}{@{}c@{}}%
0\\0\\0
\end{tabular}\endgroup%
}}\!\right]$}%
}%
\ifdim\wd\matricesbox>\halfwidth\myboxwidth=\hsize\else\myboxwidth=\halfwidth\fi
\vbox{%
\ifdim\myboxwidth=\hsize
\setbox\onelinebox=\hbox{%
\vbox{\hbox{%
$\Pi_{6,34}$ spans $L_{213.2}$%
}\hbox{%
$22|222|2\rtimes D_{2}$%
}%
}%
\hfill\copy\matricesbox
}%
\ifdim\wd\onelinebox>\myboxwidth
\hbox to \myboxwidth{%
$\Pi_{6,34}$ spans $L_{213.2}$%
\hfil
$22|222|2\rtimes D_{2}$%
}%
\box\matricesbox
\else
\hbox to \myboxwidth{%
\unhbox\onelinebox
}%
\fi
\else
\hbox to \myboxwidth{%
$\Pi_{6,34}$ spans $L_{213.2}$%
\hfil}%
\hbox to \myboxwidth{%
$22|222|2\rtimes D_{2}$%
\hfil}%
\box\matricesbox
\fi
}%
\hfill\discretionary{}{}{}%
\setbox\matricesbox=\hbox{%
{$\left[\!\llap{\phantom{%
\begingroup \smaller\smaller\smaller\begin{tabular}{@{}c@{}}%
\phantom{0}\\\phantom{0}\\\phantom{0}
\end{tabular}\endgroup%
}}\right.$}%
\begingroup \smaller\smaller\smaller\begin{tabular}{@{}c@{}}%
-1/5\\\phantom{0}\\\phantom{0}
\end{tabular}\endgroup%
\kern3pt%
\begingroup \smaller\smaller\smaller\begin{tabular}{@{}c@{}}%
\phantom{0}\\6/5\\\phantom{0}
\end{tabular}\endgroup%
\kern3pt%
\begingroup \smaller\smaller\smaller\begin{tabular}{@{}c@{}}%
\phantom{0}\\\phantom{0}\\4
\end{tabular}\endgroup%
{$\left.\llap{\phantom{%
\begingroup \smaller\smaller\smaller\begin{tabular}{@{}c@{}}%
\phantom{0}\\\phantom{0}\\\phantom{0}
\end{tabular}\endgroup%
}}\!\right]$}%
{$\left[\!\llap{\phantom{%
\begingroup \smaller\smaller\smaller\begin{tabular}{@{}c@{}}%
0\\0\\0
\end{tabular}\endgroup%
}}\right.$}%
\begingroup \smaller\smaller\smaller\begin{tabular}{@{}c@{}}%
1\\1\\0
\end{tabular}\endgroup%
\kern3pt%
\begingroup \smaller\smaller\smaller\begin{tabular}{@{}c@{}}%
12\\2\\3
\end{tabular}\endgroup%
\kern3pt%
\begingroup \smaller\smaller\smaller\begin{tabular}{@{}c@{}}%
8\\-2\\2
\end{tabular}\endgroup%
\kern3pt%
\begingroup \smaller\smaller\smaller\begin{tabular}{@{}c@{}}%
3\\-2\\0
\end{tabular}\endgroup%
{$\left.\llap{\phantom{%
\begingroup \smaller\smaller\smaller\begin{tabular}{@{}c@{}}%
0\\0\\0
\end{tabular}\endgroup%
}}\!\right]$}%
}%
\ifdim\wd\matricesbox>\halfwidth\myboxwidth=\hsize\else\myboxwidth=\halfwidth\fi
\vbox{%
\ifdim\myboxwidth=\hsize
\setbox\onelinebox=\hbox{%
\vbox{\hbox{%
$\Pi_{6,35}$ spans $L_{151.6}$%
}\hbox{%
$22|222|2\rtimes D_{2}$ (shared)%
}%
}%
\hfill\copy\matricesbox
}%
\ifdim\wd\onelinebox>\myboxwidth
\hbox to \myboxwidth{%
$\Pi_{6,35}$ spans $L_{151.6}$%
\hfil
$22|222|2\rtimes D_{2}$ (shared)%
}%
\box\matricesbox
\else
\hbox to \myboxwidth{%
\unhbox\onelinebox
}%
\fi
\else
\hbox to \myboxwidth{%
$\Pi_{6,35}$ spans $L_{151.6}$%
\hfil}%
\hbox to \myboxwidth{%
$22|222|2\rtimes D_{2}$ (shared)%
\hfil}%
\box\matricesbox
\fi
}%
\hfill\discretionary{}{}{}%
\setbox\matricesbox=\hbox{%
{$\left[\!\llap{\phantom{%
\begingroup \smaller\smaller\smaller\begin{tabular}{@{}c@{}}%
\phantom{0}\\\phantom{0}\\\phantom{0}
\end{tabular}\endgroup%
}}\right.$}%
\begingroup \smaller\smaller\smaller\begin{tabular}{@{}c@{}}%
-1/4\\\phantom{0}\\\phantom{0}
\end{tabular}\endgroup%
\kern3pt%
\begingroup \smaller\smaller\smaller\begin{tabular}{@{}c@{}}%
\phantom{0}\\5/4\\\phantom{0}
\end{tabular}\endgroup%
\kern3pt%
\begingroup \smaller\smaller\smaller\begin{tabular}{@{}c@{}}%
\phantom{0}\\\phantom{0}\\30
\end{tabular}\endgroup%
{$\left.\llap{\phantom{%
\begingroup \smaller\smaller\smaller\begin{tabular}{@{}c@{}}%
\phantom{0}\\\phantom{0}\\\phantom{0}
\end{tabular}\endgroup%
}}\!\right]$}%
{$\left[\!\llap{\phantom{%
\begingroup \smaller\smaller\smaller\begin{tabular}{@{}c@{}}%
0\\0\\0
\end{tabular}\endgroup%
}}\right.$}%
\begingroup \smaller\smaller\smaller\begin{tabular}{@{}c@{}}%
1\\-1\\0
\end{tabular}\endgroup%
\kern3pt%
\begingroup \smaller\smaller\smaller\begin{tabular}{@{}c@{}}%
10\\-2\\1
\end{tabular}\endgroup%
\kern3pt%
\begingroup \smaller\smaller\smaller\begin{tabular}{@{}c@{}}%
10\\2\\1
\end{tabular}\endgroup%
\kern3pt%
\begingroup \smaller\smaller\smaller\begin{tabular}{@{}c@{}}%
5\\3\\0
\end{tabular}\endgroup%
{$\left.\llap{\phantom{%
\begingroup \smaller\smaller\smaller\begin{tabular}{@{}c@{}}%
0\\0\\0
\end{tabular}\endgroup%
}}\!\right]$}%
}%
\ifdim\wd\matricesbox>\halfwidth\myboxwidth=\hsize\else\myboxwidth=\halfwidth\fi
\vbox{%
\ifdim\myboxwidth=\hsize
\setbox\onelinebox=\hbox{%
\vbox{\hbox{%
$\Pi_{6,36}$ spans $L_{127.7}$%
}\hbox{%
$22|224|4\rtimes D_{2}$%
}%
}%
\hfill\copy\matricesbox
}%
\ifdim\wd\onelinebox>\myboxwidth
\hbox to \myboxwidth{%
$\Pi_{6,36}$ spans $L_{127.7}$%
\hfil
$22|224|4\rtimes D_{2}$%
}%
\box\matricesbox
\else
\hbox to \myboxwidth{%
\unhbox\onelinebox
}%
\fi
\else
\hbox to \myboxwidth{%
$\Pi_{6,36}$ spans $L_{127.7}$%
\hfil}%
\hbox to \myboxwidth{%
$22|224|4\rtimes D_{2}$%
\hfil}%
\box\matricesbox
\fi
}%
\hfill\discretionary{}{}{}%
\setbox\matricesbox=\hbox{%
{$\left[\!\llap{\phantom{%
\begingroup \smaller\smaller\smaller\begin{tabular}{@{}c@{}}%
\phantom{0}\\\phantom{0}\\\phantom{0}
\end{tabular}\endgroup%
}}\right.$}%
\begingroup \smaller\smaller\smaller\begin{tabular}{@{}c@{}}%
-5/32\\\phantom{0}\\\phantom{0}
\end{tabular}\endgroup%
\kern3pt%
\begingroup \smaller\smaller\smaller\begin{tabular}{@{}c@{}}%
\phantom{0}\\9/8\\\phantom{0}
\end{tabular}\endgroup%
\kern3pt%
\begingroup \smaller\smaller\smaller\begin{tabular}{@{}c@{}}%
\phantom{0}\\\phantom{0}\\3/2
\end{tabular}\endgroup%
{$\left.\llap{\phantom{%
\begingroup \smaller\smaller\smaller\begin{tabular}{@{}c@{}}%
\phantom{0}\\\phantom{0}\\\phantom{0}
\end{tabular}\endgroup%
}}\!\right]$}%
{$\left[\!\llap{\phantom{%
\begingroup \smaller\smaller\smaller\begin{tabular}{@{}c@{}}%
0\\0\\0
\end{tabular}\endgroup%
}}\right.$}%
\begingroup \smaller\smaller\smaller\begin{tabular}{@{}c@{}}%
18\\7\\3
\end{tabular}\endgroup%
\kern3pt%
\begingroup \smaller\smaller\smaller\begin{tabular}{@{}c@{}}%
24\\4\\8
\end{tabular}\endgroup%
\kern3pt%
\begingroup \smaller\smaller\smaller\begin{tabular}{@{}c@{}}%
2\\-1\\1
\end{tabular}\endgroup%
{$\left.\llap{\phantom{%
\begingroup \smaller\smaller\smaller\begin{tabular}{@{}c@{}}%
0\\0\\0
\end{tabular}\endgroup%
}}\!\right]$}%
}%
\ifdim\wd\matricesbox>\halfwidth\myboxwidth=\hsize\else\myboxwidth=\halfwidth\fi
\vbox{%
\ifdim\myboxwidth=\hsize
\setbox\onelinebox=\hbox{%
\vbox{\hbox{%
$\Pi_{6,37}$ spans $L_{251.1}$%
}\hbox{%
$22\slashthree22\slashthree\rtimes D_{2}$ (shared)%
}%
}%
\hfill\copy\matricesbox
}%
\ifdim\wd\onelinebox>\myboxwidth
\hbox to \myboxwidth{%
$\Pi_{6,37}$ spans $L_{251.1}$%
\hfil
$22\slashthree22\slashthree\rtimes D_{2}$ (shared)%
}%
\box\matricesbox
\else
\hbox to \myboxwidth{%
\unhbox\onelinebox
}%
\fi
\else
\hbox to \myboxwidth{%
$\Pi_{6,37}$ spans $L_{251.1}$%
\hfil}%
\hbox to \myboxwidth{%
$22\slashthree22\slashthree\rtimes D_{2}$ (shared)%
\hfil}%
\box\matricesbox
\fi
}%
\hfill\discretionary{}{}{}%
\setbox\matricesbox=\hbox{%
{$\left[\!\llap{\phantom{%
\begingroup \smaller\smaller\smaller\begin{tabular}{@{}c@{}}%
\phantom{0}\\\phantom{0}\\\phantom{0}
\end{tabular}\endgroup%
}}\right.$}%
\begingroup \smaller\smaller\smaller\begin{tabular}{@{}c@{}}%
-3/4\\\phantom{0}\\\phantom{0}
\end{tabular}\endgroup%
\kern3pt%
\begingroup \smaller\smaller\smaller\begin{tabular}{@{}c@{}}%
\phantom{0}\\3/4\\\phantom{0}
\end{tabular}\endgroup%
\kern3pt%
\begingroup \smaller\smaller\smaller\begin{tabular}{@{}c@{}}%
\phantom{0}\\\phantom{0}\\1
\end{tabular}\endgroup%
{$\left.\llap{\phantom{%
\begingroup \smaller\smaller\smaller\begin{tabular}{@{}c@{}}%
\phantom{0}\\\phantom{0}\\\phantom{0}
\end{tabular}\endgroup%
}}\!\right]$}%
{$\left[\!\llap{\phantom{%
\begingroup \smaller\smaller\smaller\begin{tabular}{@{}c@{}}%
0\\0\\0
\end{tabular}\endgroup%
}}\right.$}%
\begingroup \smaller\smaller\smaller\begin{tabular}{@{}c@{}}%
4\\-4\\-2
\end{tabular}\endgroup%
\kern3pt%
\begingroup \smaller\smaller\smaller\begin{tabular}{@{}c@{}}%
3\\-1\\-3
\end{tabular}\endgroup%
\kern3pt%
\begingroup \smaller\smaller\smaller\begin{tabular}{@{}c@{}}%
1\\1\\-1
\end{tabular}\endgroup%
{$\left.\llap{\phantom{%
\begingroup \smaller\smaller\smaller\begin{tabular}{@{}c@{}}%
0\\0\\0
\end{tabular}\endgroup%
}}\!\right]$}%
}%
\ifdim\wd\matricesbox>\halfwidth\myboxwidth=\hsize\else\myboxwidth=\halfwidth\fi
\vbox{%
\ifdim\myboxwidth=\hsize
\setbox\onelinebox=\hbox{%
\vbox{\hbox{%
$\Pi_{6,38}$ spans $L_{144.5}$%
}\hbox{%
$22\slashinfty22\slashinfty\rtimes D_{2}$ (shared)%
}%
}%
\hfill\copy\matricesbox
}%
\ifdim\wd\onelinebox>\myboxwidth
\hbox to \myboxwidth{%
$\Pi_{6,38}$ spans $L_{144.5}$%
\hfil
$22\slashinfty22\slashinfty\rtimes D_{2}$ (shared)%
}%
\box\matricesbox
\else
\hbox to \myboxwidth{%
\unhbox\onelinebox
}%
\fi
\else
\hbox to \myboxwidth{%
$\Pi_{6,38}$ spans $L_{144.5}$%
\hfil}%
\hbox to \myboxwidth{%
$22\slashinfty22\slashinfty\rtimes D_{2}$ (shared)%
\hfil}%
\box\matricesbox
\fi
}%
\hfill\discretionary{}{}{}%
\setbox\matricesbox=\hbox{%
{$\left[\!\llap{\phantom{%
\begingroup \smaller\smaller\smaller\begin{tabular}{@{}c@{}}%
\phantom{0}\\\phantom{0}\\\phantom{0}
\end{tabular}\endgroup%
}}\right.$}%
\begingroup \smaller\smaller\smaller\begin{tabular}{@{}c@{}}%
-1/8\\\phantom{0}\\\phantom{0}
\end{tabular}\endgroup%
\kern3pt%
\begingroup \smaller\smaller\smaller\begin{tabular}{@{}c@{}}%
\phantom{0}\\5/2\\\phantom{0}
\end{tabular}\endgroup%
\kern3pt%
\begingroup \smaller\smaller\smaller\begin{tabular}{@{}c@{}}%
\phantom{0}\\\phantom{0}\\120
\end{tabular}\endgroup%
{$\left.\llap{\phantom{%
\begingroup \smaller\smaller\smaller\begin{tabular}{@{}c@{}}%
\phantom{0}\\\phantom{0}\\\phantom{0}
\end{tabular}\endgroup%
}}\!\right]$}%
{$\left[\!\llap{\phantom{%
\begingroup \smaller\smaller\smaller\begin{tabular}{@{}c@{}}%
0\\0\\0
\end{tabular}\endgroup%
}}\right.$}%
\begingroup \smaller\smaller\smaller\begin{tabular}{@{}c@{}}%
10\\-3\\0
\end{tabular}\endgroup%
\kern3pt%
\begingroup \smaller\smaller\smaller\begin{tabular}{@{}c@{}}%
30\\-3\\-1
\end{tabular}\endgroup%
\kern3pt%
\begingroup \smaller\smaller\smaller\begin{tabular}{@{}c@{}}%
30\\3\\-1
\end{tabular}\endgroup%
\kern3pt%
\begingroup \smaller\smaller\smaller\begin{tabular}{@{}c@{}}%
2\\1\\0
\end{tabular}\endgroup%
{$\left.\llap{\phantom{%
\begingroup \smaller\smaller\smaller\begin{tabular}{@{}c@{}}%
0\\0\\0
\end{tabular}\endgroup%
}}\!\right]$}%
}%
\ifdim\wd\matricesbox>\halfwidth\myboxwidth=\hsize\else\myboxwidth=\halfwidth\fi
\vbox{%
\ifdim\myboxwidth=\hsize
\setbox\onelinebox=\hbox{%
\vbox{\hbox{%
$\Pi_{6,39}$ spans $L_{31.10}$%
}\hbox{%
$36|632|2\rtimes D_{2}$%
}%
}%
\hfill\copy\matricesbox
}%
\ifdim\wd\onelinebox>\myboxwidth
\hbox to \myboxwidth{%
$\Pi_{6,39}$ spans $L_{31.10}$%
\hfil
$36|632|2\rtimes D_{2}$%
}%
\box\matricesbox
\else
\hbox to \myboxwidth{%
\unhbox\onelinebox
}%
\fi
\else
\hbox to \myboxwidth{%
$\Pi_{6,39}$ spans $L_{31.10}$%
\hfil}%
\hbox to \myboxwidth{%
$36|632|2\rtimes D_{2}$%
\hfil}%
\box\matricesbox
\fi
}%
\hfill\discretionary{}{}{}%
\setbox\matricesbox=\hbox{%
{$\left[\!\llap{\phantom{%
\begingroup \smaller\smaller\smaller\begin{tabular}{@{}c@{}}%
\phantom{0}\\\phantom{0}\\\phantom{0}
\end{tabular}\endgroup%
}}\right.$}%
\begingroup \smaller\smaller\smaller\begin{tabular}{@{}c@{}}%
-1/8\\\phantom{0}\\\phantom{0}
\end{tabular}\endgroup%
\kern3pt%
\begingroup \smaller\smaller\smaller\begin{tabular}{@{}c@{}}%
\phantom{0}\\21/2\\\phantom{0}
\end{tabular}\endgroup%
\kern3pt%
\begingroup \smaller\smaller\smaller\begin{tabular}{@{}c@{}}%
\phantom{0}\\\phantom{0}\\28
\end{tabular}\endgroup%
{$\left.\llap{\phantom{%
\begingroup \smaller\smaller\smaller\begin{tabular}{@{}c@{}}%
\phantom{0}\\\phantom{0}\\\phantom{0}
\end{tabular}\endgroup%
}}\!\right]$}%
{$\left[\!\llap{\phantom{%
\begingroup \smaller\smaller\smaller\begin{tabular}{@{}c@{}}%
0\\0\\0
\end{tabular}\endgroup%
}}\right.$}%
\begingroup \smaller\smaller\smaller\begin{tabular}{@{}c@{}}%
42\\5\\0
\end{tabular}\endgroup%
\kern3pt%
\begingroup \smaller\smaller\smaller\begin{tabular}{@{}c@{}}%
14\\1\\-1
\end{tabular}\endgroup%
\kern3pt%
\begingroup \smaller\smaller\smaller\begin{tabular}{@{}c@{}}%
14\\-1\\-1
\end{tabular}\endgroup%
\kern3pt%
\begingroup \smaller\smaller\smaller\begin{tabular}{@{}c@{}}%
6\\-1\\0
\end{tabular}\endgroup%
{$\left.\llap{\phantom{%
\begingroup \smaller\smaller\smaller\begin{tabular}{@{}c@{}}%
0\\0\\0
\end{tabular}\endgroup%
}}\!\right]$}%
}%
\ifdim\wd\matricesbox>\halfwidth\myboxwidth=\hsize\else\myboxwidth=\halfwidth\fi
\vbox{%
\ifdim\myboxwidth=\hsize
\setbox\onelinebox=\hbox{%
\vbox{\hbox{%
$\Pi_{6,40}$ spans $L_{22.8}$%
}\hbox{%
$36|632|2\rtimes D_{2}$%
}%
}%
\hfill\copy\matricesbox
}%
\ifdim\wd\onelinebox>\myboxwidth
\hbox to \myboxwidth{%
$\Pi_{6,40}$ spans $L_{22.8}$%
\hfil
$36|632|2\rtimes D_{2}$%
}%
\box\matricesbox
\else
\hbox to \myboxwidth{%
\unhbox\onelinebox
}%
\fi
\else
\hbox to \myboxwidth{%
$\Pi_{6,40}$ spans $L_{22.8}$%
\hfil}%
\hbox to \myboxwidth{%
$36|632|2\rtimes D_{2}$%
\hfil}%
\box\matricesbox
\fi
}%
\hfill\discretionary{}{}{}%
\setbox\matricesbox=\hbox{%
{$\left[\!\llap{\phantom{%
\begingroup \smaller\smaller\smaller\begin{tabular}{@{}c@{}}%
\phantom{0}\\\phantom{0}\\\phantom{0}
\end{tabular}\endgroup%
}}\right.$}%
\begingroup \smaller\smaller\smaller\begin{tabular}{@{}c@{}}%
-1\\\phantom{0}\\\phantom{0}
\end{tabular}\endgroup%
\kern3pt%
\begingroup \smaller\smaller\smaller\begin{tabular}{@{}c@{}}%
\phantom{0}\\2\\\phantom{0}
\end{tabular}\endgroup%
\kern3pt%
\begingroup \smaller\smaller\smaller\begin{tabular}{@{}c@{}}%
\phantom{0}\\\phantom{0}\\4
\end{tabular}\endgroup%
{$\left.\llap{\phantom{%
\begingroup \smaller\smaller\smaller\begin{tabular}{@{}c@{}}%
\phantom{0}\\\phantom{0}\\\phantom{0}
\end{tabular}\endgroup%
}}\!\right]$}%
{$\left[\!\llap{\phantom{%
\begingroup \smaller\smaller\smaller\begin{tabular}{@{}c@{}}%
0\\0\\0
\end{tabular}\endgroup%
}}\right.$}%
\begingroup \smaller\smaller\smaller\begin{tabular}{@{}c@{}}%
8\\6\\0
\end{tabular}\endgroup%
\kern3pt%
\begingroup \smaller\smaller\smaller\begin{tabular}{@{}c@{}}%
2\\1\\1
\end{tabular}\endgroup%
\kern3pt%
\begingroup \smaller\smaller\smaller\begin{tabular}{@{}c@{}}%
2\\-1\\1
\end{tabular}\endgroup%
\kern3pt%
\begingroup \smaller\smaller\smaller\begin{tabular}{@{}c@{}}%
1\\-1\\0
\end{tabular}\endgroup%
{$\left.\llap{\phantom{%
\begingroup \smaller\smaller\smaller\begin{tabular}{@{}c@{}}%
0\\0\\0
\end{tabular}\endgroup%
}}\!\right]$}%
}%
\ifdim\wd\matricesbox>\halfwidth\myboxwidth=\hsize\else\myboxwidth=\halfwidth\fi
\vbox{%
\ifdim\myboxwidth=\hsize
\setbox\onelinebox=\hbox{%
\vbox{\hbox{%
$\Pi_{6,41}$ spans $L_{141.3}$%
}\hbox{%
$\infty|\infty\infty2|2\infty\rtimes D_{2}$%
}%
}%
\hfill\copy\matricesbox
}%
\ifdim\wd\onelinebox>\myboxwidth
\hbox to \myboxwidth{%
$\Pi_{6,41}$ spans $L_{141.3}$%
\hfil
$\infty|\infty\infty2|2\infty\rtimes D_{2}$%
}%
\box\matricesbox
\else
\hbox to \myboxwidth{%
\unhbox\onelinebox
}%
\fi
\else
\hbox to \myboxwidth{%
$\Pi_{6,41}$ spans $L_{141.3}$%
\hfil}%
\hbox to \myboxwidth{%
$\infty|\infty\infty2|2\infty\rtimes D_{2}$%
\hfil}%
\box\matricesbox
\fi
}%
\hfill\discretionary{}{}{}%
\setbox\matricesbox=\hbox{%
{$\left[\!\llap{\phantom{%
\begingroup \smaller\smaller\smaller\begin{tabular}{@{}c@{}}%
\phantom{0}\\\phantom{0}\\\phantom{0}
\end{tabular}\endgroup%
}}\right.$}%
\begingroup \smaller\smaller\smaller\begin{tabular}{@{}c@{}}%
-1/4\\\phantom{0}\\\phantom{0}
\end{tabular}\endgroup%
\kern3pt%
\begingroup \smaller\smaller\smaller\begin{tabular}{@{}c@{}}%
\phantom{0}\\5/4\\\phantom{0}
\end{tabular}\endgroup%
\kern3pt%
\begingroup \smaller\smaller\smaller\begin{tabular}{@{}c@{}}%
\phantom{0}\\\phantom{0}\\100
\end{tabular}\endgroup%
{$\left.\llap{\phantom{%
\begingroup \smaller\smaller\smaller\begin{tabular}{@{}c@{}}%
\phantom{0}\\\phantom{0}\\\phantom{0}
\end{tabular}\endgroup%
}}\!\right]$}%
{$\left[\!\llap{\phantom{%
\begingroup \smaller\smaller\smaller\begin{tabular}{@{}c@{}}%
0\\0\\0
\end{tabular}\endgroup%
}}\right.$}%
\begingroup \smaller\smaller\smaller\begin{tabular}{@{}c@{}}%
5\\-3\\0
\end{tabular}\endgroup%
\kern3pt%
\begingroup \smaller\smaller\smaller\begin{tabular}{@{}c@{}}%
20\\-4\\-1
\end{tabular}\endgroup%
\kern3pt%
\begingroup \smaller\smaller\smaller\begin{tabular}{@{}c@{}}%
20\\4\\-1
\end{tabular}\endgroup%
\kern3pt%
\begingroup \smaller\smaller\smaller\begin{tabular}{@{}c@{}}%
1\\1\\0
\end{tabular}\endgroup%
{$\left.\llap{\phantom{%
\begingroup \smaller\smaller\smaller\begin{tabular}{@{}c@{}}%
0\\0\\0
\end{tabular}\endgroup%
}}\!\right]$}%
}%
\ifdim\wd\matricesbox>\halfwidth\myboxwidth=\hsize\else\myboxwidth=\halfwidth\fi
\vbox{%
\ifdim\myboxwidth=\hsize
\setbox\onelinebox=\hbox{%
\vbox{\hbox{%
$\Pi_{6,42}$ spans $L_{149.21}$%
}\hbox{%
$|\infty\infty2|2\infty\infty\rtimes D_{2}$%
}%
}%
\hfill\copy\matricesbox
}%
\ifdim\wd\onelinebox>\myboxwidth
\hbox to \myboxwidth{%
$\Pi_{6,42}$ spans $L_{149.21}$%
\hfil
$|\infty\infty2|2\infty\infty\rtimes D_{2}$%
}%
\box\matricesbox
\else
\hbox to \myboxwidth{%
\unhbox\onelinebox
}%
\fi
\else
\hbox to \myboxwidth{%
$\Pi_{6,42}$ spans $L_{149.21}$%
\hfil}%
\hbox to \myboxwidth{%
$|\infty\infty2|2\infty\infty\rtimes D_{2}$%
\hfil}%
\box\matricesbox
\fi
}%
\hfill\discretionary{}{}{}%
\setbox\matricesbox=\hbox{%
{$\left[\!\llap{\phantom{%
\begingroup \smaller\smaller\smaller\begin{tabular}{@{}c@{}}%
\phantom{0}\\\phantom{0}\\\phantom{0}
\end{tabular}\endgroup%
}}\right.$}%
\begingroup \smaller\smaller\smaller\begin{tabular}{@{}c@{}}%
-8/9\\\phantom{0}\\\phantom{0}
\end{tabular}\endgroup%
\kern3pt%
\begingroup \smaller\smaller\smaller\begin{tabular}{@{}c@{}}%
\phantom{0}\\1/18\\\phantom{0}
\end{tabular}\endgroup%
\kern3pt%
\begingroup \smaller\smaller\smaller\begin{tabular}{@{}c@{}}%
\phantom{0}\\\phantom{0}\\1/2
\end{tabular}\endgroup%
{$\left.\llap{\phantom{%
\begingroup \smaller\smaller\smaller\begin{tabular}{@{}c@{}}%
\phantom{0}\\\phantom{0}\\\phantom{0}
\end{tabular}\endgroup%
}}\!\right]$}%
{$\left[\!\llap{\phantom{%
\begingroup \smaller\smaller\smaller\begin{tabular}{@{}c@{}}%
0\\0\\0
\end{tabular}\endgroup%
}}\right.$}%
\begingroup \smaller\smaller\smaller\begin{tabular}{@{}c@{}}%
2\\8\\-2
\end{tabular}\endgroup%
\kern3pt%
\begingroup \smaller\smaller\smaller\begin{tabular}{@{}c@{}}%
4\\-2\\-6
\end{tabular}\endgroup%
\kern3pt%
\begingroup \smaller\smaller\smaller\begin{tabular}{@{}c@{}}%
1\\-5\\-1
\end{tabular}\endgroup%
{$\left.\llap{\phantom{%
\begingroup \smaller\smaller\smaller\begin{tabular}{@{}c@{}}%
0\\0\\0
\end{tabular}\endgroup%
}}\!\right]$}%
}%
\ifdim\wd\matricesbox>\halfwidth\myboxwidth=\hsize\else\myboxwidth=\halfwidth\fi
\vbox{%
\ifdim\myboxwidth=\hsize
\setbox\onelinebox=\hbox{%
\vbox{\hbox{%
$\Pi_{6,43}$ spans $L_{159.1}$%
}\hbox{%
$4\slashinfty42\slashtwo2\rtimes D_{2}$%
}%
}%
\hfill\copy\matricesbox
}%
\ifdim\wd\onelinebox>\myboxwidth
\hbox to \myboxwidth{%
$\Pi_{6,43}$ spans $L_{159.1}$%
\hfil
$4\slashinfty42\slashtwo2\rtimes D_{2}$%
}%
\box\matricesbox
\else
\hbox to \myboxwidth{%
\unhbox\onelinebox
}%
\fi
\else
\hbox to \myboxwidth{%
$\Pi_{6,43}$ spans $L_{159.1}$%
\hfil}%
\hbox to \myboxwidth{%
$4\slashinfty42\slashtwo2\rtimes D_{2}$%
\hfil}%
\box\matricesbox
\fi
}%
\hfill\discretionary{}{}{}%
\setbox\matricesbox=\hbox{%
{$\left[\!\llap{\phantom{%
\begingroup \smaller\smaller\smaller\begin{tabular}{@{}c@{}}%
\phantom{0}\\\phantom{0}\\\phantom{0}
\end{tabular}\endgroup%
}}\right.$}%
\begingroup \smaller\smaller\smaller\begin{tabular}{@{}c@{}}%
-1/2\\\phantom{0}\\\phantom{0}
\end{tabular}\endgroup%
\kern3pt%
\begingroup \smaller\smaller\smaller\begin{tabular}{@{}c@{}}%
\phantom{0}\\3/2\\\phantom{0}
\end{tabular}\endgroup%
\kern3pt%
\begingroup \smaller\smaller\smaller\begin{tabular}{@{}c@{}}%
\phantom{0}\\\phantom{0}\\18
\end{tabular}\endgroup%
{$\left.\llap{\phantom{%
\begingroup \smaller\smaller\smaller\begin{tabular}{@{}c@{}}%
\phantom{0}\\\phantom{0}\\\phantom{0}
\end{tabular}\endgroup%
}}\!\right]$}%
{$\left[\!\llap{\phantom{%
\begingroup \smaller\smaller\smaller\begin{tabular}{@{}c@{}}%
0\\0\\0
\end{tabular}\endgroup%
}}\right.$}%
\begingroup \smaller\smaller\smaller\begin{tabular}{@{}c@{}}%
6\\-4\\0
\end{tabular}\endgroup%
\kern3pt%
\begingroup \smaller\smaller\smaller\begin{tabular}{@{}c@{}}%
6\\-2\\1
\end{tabular}\endgroup%
\kern3pt%
\begingroup \smaller\smaller\smaller\begin{tabular}{@{}c@{}}%
6\\2\\1
\end{tabular}\endgroup%
\kern3pt%
\begingroup \smaller\smaller\smaller\begin{tabular}{@{}c@{}}%
1\\1\\0
\end{tabular}\endgroup%
{$\left.\llap{\phantom{%
\begingroup \smaller\smaller\smaller\begin{tabular}{@{}c@{}}%
0\\0\\0
\end{tabular}\endgroup%
}}\!\right]$}%
}%
\ifdim\wd\matricesbox>\halfwidth\myboxwidth=\hsize\else\myboxwidth=\halfwidth\fi
\vbox{%
\ifdim\myboxwidth=\hsize
\setbox\onelinebox=\hbox{%
\vbox{\hbox{%
$\Pi_{6,44}$ spans $L_{4.22}$%
}\hbox{%
$|\infty\infty2|2\infty\infty\rtimes D_{2}$%
}%
}%
\hfill\copy\matricesbox
}%
\ifdim\wd\onelinebox>\myboxwidth
\hbox to \myboxwidth{%
$\Pi_{6,44}$ spans $L_{4.22}$%
\hfil
$|\infty\infty2|2\infty\infty\rtimes D_{2}$%
}%
\box\matricesbox
\else
\hbox to \myboxwidth{%
\unhbox\onelinebox
}%
\fi
\else
\hbox to \myboxwidth{%
$\Pi_{6,44}$ spans $L_{4.22}$%
\hfil}%
\hbox to \myboxwidth{%
$|\infty\infty2|2\infty\infty\rtimes D_{2}$%
\hfil}%
\box\matricesbox
\fi
}%
\hfill\discretionary{}{}{}%
\setbox\matricesbox=\hbox{%
{$\left[\!\llap{\phantom{%
\begingroup \smaller\smaller\smaller\begin{tabular}{@{}c@{}}%
\phantom{0}\\\phantom{0}\\\phantom{0}
\end{tabular}\endgroup%
}}\right.$}%
\begingroup \smaller\smaller\smaller\begin{tabular}{@{}c@{}}%
-1\\\phantom{0}\\\phantom{0}
\end{tabular}\endgroup%
\kern3pt%
\begingroup \smaller\smaller\smaller\begin{tabular}{@{}c@{}}%
\phantom{0}\\1\\\phantom{0}
\end{tabular}\endgroup%
\kern3pt%
\begingroup \smaller\smaller\smaller\begin{tabular}{@{}c@{}}%
\phantom{0}\\\phantom{0}\\1
\end{tabular}\endgroup%
{$\left.\llap{\phantom{%
\begingroup \smaller\smaller\smaller\begin{tabular}{@{}c@{}}%
\phantom{0}\\\phantom{0}\\\phantom{0}
\end{tabular}\endgroup%
}}\!\right]$}%
{$\left[\!\llap{\phantom{%
\begingroup \smaller\smaller\smaller\begin{tabular}{@{}c@{}}%
0\\0\\0
\end{tabular}\endgroup%
}}\right.$}%
\begingroup \smaller\smaller\smaller\begin{tabular}{@{}c@{}}%
1\\1\\-1
\end{tabular}\endgroup%
\kern3pt%
\begingroup \smaller\smaller\smaller\begin{tabular}{@{}c@{}}%
4\\-2\\-4
\end{tabular}\endgroup%
\kern3pt%
\begingroup \smaller\smaller\smaller\begin{tabular}{@{}c@{}}%
4\\-4\\-2
\end{tabular}\endgroup%
{$\left.\llap{\phantom{%
\begingroup \smaller\smaller\smaller\begin{tabular}{@{}c@{}}%
0\\0\\0
\end{tabular}\endgroup%
}}\!\right]$}%
}%
\ifdim\wd\matricesbox>\halfwidth\myboxwidth=\hsize\else\myboxwidth=\halfwidth\fi
\vbox{%
\ifdim\myboxwidth=\hsize
\setbox\onelinebox=\hbox{%
\vbox{\hbox{%
$\Pi_{6,45}=\hbox{GN}_{38}$ spans $L_{1.9}$%
}\hbox{%
$\slashinfty\infty2\slashinfty2\infty\rtimes D_{2}$%
}%
}%
\hfill\copy\matricesbox
}%
\ifdim\wd\onelinebox>\myboxwidth
\hbox to \myboxwidth{%
$\Pi_{6,45}=\hbox{GN}_{38}$ spans $L_{1.9}$%
\hfil
$\slashinfty\infty2\slashinfty2\infty\rtimes D_{2}$%
}%
\box\matricesbox
\else
\hbox to \myboxwidth{%
\unhbox\onelinebox
}%
\fi
\else
\hbox to \myboxwidth{%
$\Pi_{6,45}=\hbox{GN}_{38}$ spans $L_{1.9}$%
\hfil}%
\hbox to \myboxwidth{%
$\slashinfty\infty2\slashinfty2\infty\rtimes D_{2}$%
\hfil}%
\box\matricesbox
\fi
}%
\hfill\discretionary{}{}{}%
\setbox\matricesbox=\hbox{%
{$\left[\!\llap{\phantom{%
\begingroup \smaller\smaller\smaller\begin{tabular}{@{}c@{}}%
\phantom{0}\\\phantom{0}\\\phantom{0}
\end{tabular}\endgroup%
}}\right.$}%
\begingroup \smaller\smaller\smaller\begin{tabular}{@{}c@{}}%
-1/5\\\phantom{0}\\\phantom{0}
\end{tabular}\endgroup%
\kern3pt%
\begingroup \smaller\smaller\smaller\begin{tabular}{@{}c@{}}%
\phantom{0}\\6/5\\\phantom{0}
\end{tabular}\endgroup%
\kern3pt%
\begingroup \smaller\smaller\smaller\begin{tabular}{@{}c@{}}%
\phantom{0}\\\phantom{0}\\30
\end{tabular}\endgroup%
{$\left.\llap{\phantom{%
\begingroup \smaller\smaller\smaller\begin{tabular}{@{}c@{}}%
\phantom{0}\\\phantom{0}\\\phantom{0}
\end{tabular}\endgroup%
}}\!\right]$}%
{$\left[\!\llap{\phantom{%
\begingroup \smaller\smaller\smaller\begin{tabular}{@{}c@{}}%
0\\0\\0
\end{tabular}\endgroup%
}}\right.$}%
\begingroup \smaller\smaller\smaller\begin{tabular}{@{}c@{}}%
1\\1\\0
\end{tabular}\endgroup%
\kern3pt%
\begingroup \smaller\smaller\smaller\begin{tabular}{@{}c@{}}%
24\\4\\2
\end{tabular}\endgroup%
\kern3pt%
\begingroup \smaller\smaller\smaller\begin{tabular}{@{}c@{}}%
12\\-3\\1
\end{tabular}\endgroup%
\kern3pt%
\begingroup \smaller\smaller\smaller\begin{tabular}{@{}c@{}}%
3\\-2\\0
\end{tabular}\endgroup%
{$\left.\llap{\phantom{%
\begingroup \smaller\smaller\smaller\begin{tabular}{@{}c@{}}%
0\\0\\0
\end{tabular}\endgroup%
}}\!\right]$}%
}%
\ifdim\wd\matricesbox>\halfwidth\myboxwidth=\hsize\else\myboxwidth=\halfwidth\fi
\vbox{%
\ifdim\myboxwidth=\hsize
\setbox\onelinebox=\hbox{%
\vbox{\hbox{%
$\Pi_{6,46}$ spans $L_{127.9}$%
}\hbox{%
$42|242|2\rtimes D_{2}$%
}%
}%
\hfill\copy\matricesbox
}%
\ifdim\wd\onelinebox>\myboxwidth
\hbox to \myboxwidth{%
$\Pi_{6,46}$ spans $L_{127.9}$%
\hfil
$42|242|2\rtimes D_{2}$%
}%
\box\matricesbox
\else
\hbox to \myboxwidth{%
\unhbox\onelinebox
}%
\fi
\else
\hbox to \myboxwidth{%
$\Pi_{6,46}$ spans $L_{127.9}$%
\hfil}%
\hbox to \myboxwidth{%
$42|242|2\rtimes D_{2}$%
\hfil}%
\box\matricesbox
\fi
}%
\hfill\discretionary{}{}{}%
\setbox\matricesbox=\hbox{%
{$\left[\!\llap{\phantom{%
\begingroup \smaller\smaller\smaller\begin{tabular}{@{}c@{}}%
\phantom{0}\\\phantom{0}\\\phantom{0}
\end{tabular}\endgroup%
}}\right.$}%
\begingroup \smaller\smaller\smaller\begin{tabular}{@{}c@{}}%
-1/2\\\phantom{0}\\\phantom{0}
\end{tabular}\endgroup%
\kern3pt%
\begingroup \smaller\smaller\smaller\begin{tabular}{@{}c@{}}%
\phantom{0}\\3/2\\-1/2
\end{tabular}\endgroup%
\kern3pt%
\begingroup \smaller\smaller\smaller\begin{tabular}{@{}c@{}}%
\phantom{0}\\-1/2\\7/2
\end{tabular}\endgroup%
{$\left.\llap{\phantom{%
\begingroup \smaller\smaller\smaller\begin{tabular}{@{}c@{}}%
\phantom{0}\\\phantom{0}\\\phantom{0}
\end{tabular}\endgroup%
}}\!\right]$}%
{$\left[\!\llap{\phantom{%
\begingroup \smaller\smaller\smaller\begin{tabular}{@{}c@{}}%
0\\0\\0
\end{tabular}\endgroup%
}}\right.$}%
\begingroup \smaller\smaller\smaller\begin{tabular}{@{}c@{}}%
2\\-1\\-1
\end{tabular}\endgroup%
\kern3pt%
\begingroup \smaller\smaller\smaller\begin{tabular}{@{}c@{}}%
5\\1\\-2
\end{tabular}\endgroup%
\kern3pt%
\begingroup \smaller\smaller\smaller\begin{tabular}{@{}c@{}}%
1\\1\\0
\end{tabular}\endgroup%
{$\left.\llap{\phantom{%
\begingroup \smaller\smaller\smaller\begin{tabular}{@{}c@{}}%
0\\0\\0
\end{tabular}\endgroup%
}}\!\right]$}%
}%
\ifdim\wd\matricesbox>\halfwidth\myboxwidth=\hsize\else\myboxwidth=\halfwidth\fi
\vbox{%
\ifdim\myboxwidth=\hsize
\setbox\onelinebox=\hbox{%
\vbox{\hbox{%
$\Pi_{6,47}$ spans $L_{124.8}$%
}\hbox{%
$222222\rtimes C_{2}$%
}%
}%
\hfill\copy\matricesbox
}%
\ifdim\wd\onelinebox>\myboxwidth
\hbox to \myboxwidth{%
$\Pi_{6,47}$ spans $L_{124.8}$%
\hfil
$222222\rtimes C_{2}$%
}%
\box\matricesbox
\else
\hbox to \myboxwidth{%
\unhbox\onelinebox
}%
\fi
\else
\hbox to \myboxwidth{%
$\Pi_{6,47}$ spans $L_{124.8}$%
\hfil}%
\hbox to \myboxwidth{%
$222222\rtimes C_{2}$%
\hfil}%
\box\matricesbox
\fi
}%
\hfill\discretionary{}{}{}%
\setbox\matricesbox=\hbox{%
{$\left[\!\llap{\phantom{%
\begingroup \smaller\smaller\smaller\begin{tabular}{@{}c@{}}%
\phantom{0}\\\phantom{0}\\\phantom{0}
\end{tabular}\endgroup%
}}\right.$}%
\begingroup \smaller\smaller\smaller\begin{tabular}{@{}c@{}}%
-1/8\\\phantom{0}\\\phantom{0}
\end{tabular}\endgroup%
\kern3pt%
\begingroup \smaller\smaller\smaller\begin{tabular}{@{}c@{}}%
\phantom{0}\\5/2\\-1/2
\end{tabular}\endgroup%
\kern3pt%
\begingroup \smaller\smaller\smaller\begin{tabular}{@{}c@{}}%
\phantom{0}\\-1/2\\53/2
\end{tabular}\endgroup%
{$\left.\llap{\phantom{%
\begingroup \smaller\smaller\smaller\begin{tabular}{@{}c@{}}%
\phantom{0}\\\phantom{0}\\\phantom{0}
\end{tabular}\endgroup%
}}\!\right]$}%
{$\left[\!\llap{\phantom{%
\begingroup \smaller\smaller\smaller\begin{tabular}{@{}c@{}}%
0\\0\\0
\end{tabular}\endgroup%
}}\right.$}%
\begingroup \smaller\smaller\smaller\begin{tabular}{@{}c@{}}%
12\\1\\-1
\end{tabular}\endgroup%
\kern3pt%
\begingroup \smaller\smaller\smaller\begin{tabular}{@{}c@{}}%
44\\-5\\-3
\end{tabular}\endgroup%
\kern3pt%
\begingroup \smaller\smaller\smaller\begin{tabular}{@{}c@{}}%
2\\-1\\0
\end{tabular}\endgroup%
{$\left.\llap{\phantom{%
\begingroup \smaller\smaller\smaller\begin{tabular}{@{}c@{}}%
0\\0\\0
\end{tabular}\endgroup%
}}\!\right]$}%
}%
\ifdim\wd\matricesbox>\halfwidth\myboxwidth=\hsize\else\myboxwidth=\halfwidth\fi
\vbox{%
\ifdim\myboxwidth=\hsize
\setbox\onelinebox=\hbox{%
\vbox{\hbox{%
$\Pi_{6,48}$ spans $L_{35.3}$%
}\hbox{%
$222222\rtimes C_{2}$%
}%
}%
\hfill\copy\matricesbox
}%
\ifdim\wd\onelinebox>\myboxwidth
\hbox to \myboxwidth{%
$\Pi_{6,48}$ spans $L_{35.3}$%
\hfil
$222222\rtimes C_{2}$%
}%
\box\matricesbox
\else
\hbox to \myboxwidth{%
\unhbox\onelinebox
}%
\fi
\else
\hbox to \myboxwidth{%
$\Pi_{6,48}$ spans $L_{35.3}$%
\hfil}%
\hbox to \myboxwidth{%
$222222\rtimes C_{2}$%
\hfil}%
\box\matricesbox
\fi
}%
\hfill\discretionary{}{}{}%
\setbox\matricesbox=\hbox{%
{$\left[\!\llap{\phantom{%
\begingroup \smaller\smaller\smaller\begin{tabular}{@{}c@{}}%
\phantom{0}\\\phantom{0}\\\phantom{0}
\end{tabular}\endgroup%
}}\right.$}%
\begingroup \smaller\smaller\smaller\begin{tabular}{@{}c@{}}%
-1/8\\\phantom{0}\\\phantom{0}
\end{tabular}\endgroup%
\kern3pt%
\begingroup \smaller\smaller\smaller\begin{tabular}{@{}c@{}}%
\phantom{0}\\5/2\\-1
\end{tabular}\endgroup%
\kern3pt%
\begingroup \smaller\smaller\smaller\begin{tabular}{@{}c@{}}%
\phantom{0}\\-1\\34
\end{tabular}\endgroup%
{$\left.\llap{\phantom{%
\begingroup \smaller\smaller\smaller\begin{tabular}{@{}c@{}}%
\phantom{0}\\\phantom{0}\\\phantom{0}
\end{tabular}\endgroup%
}}\!\right]$}%
{$\left[\!\llap{\phantom{%
\begingroup \smaller\smaller\smaller\begin{tabular}{@{}c@{}}%
0\\0\\0
\end{tabular}\endgroup%
}}\right.$}%
\begingroup \smaller\smaller\smaller\begin{tabular}{@{}c@{}}%
14\\-1\\1
\end{tabular}\endgroup%
\kern3pt%
\begingroup \smaller\smaller\smaller\begin{tabular}{@{}c@{}}%
32\\4\\2
\end{tabular}\endgroup%
\kern3pt%
\begingroup \smaller\smaller\smaller\begin{tabular}{@{}c@{}}%
2\\1\\0
\end{tabular}\endgroup%
{$\left.\llap{\phantom{%
\begingroup \smaller\smaller\smaller\begin{tabular}{@{}c@{}}%
0\\0\\0
\end{tabular}\endgroup%
}}\!\right]$}%
}%
\ifdim\wd\matricesbox>\halfwidth\myboxwidth=\hsize\else\myboxwidth=\halfwidth\fi
\vbox{%
\ifdim\myboxwidth=\hsize
\setbox\onelinebox=\hbox{%
\vbox{\hbox{%
$\Pi_{6,49}$ spans $L_{45.3}$%
}\hbox{%
$222222\rtimes C_{2}$%
}%
}%
\hfill\copy\matricesbox
}%
\ifdim\wd\onelinebox>\myboxwidth
\hbox to \myboxwidth{%
$\Pi_{6,49}$ spans $L_{45.3}$%
\hfil
$222222\rtimes C_{2}$%
}%
\box\matricesbox
\else
\hbox to \myboxwidth{%
\unhbox\onelinebox
}%
\fi
\else
\hbox to \myboxwidth{%
$\Pi_{6,49}$ spans $L_{45.3}$%
\hfil}%
\hbox to \myboxwidth{%
$222222\rtimes C_{2}$%
\hfil}%
\box\matricesbox
\fi
}%
\hfill\discretionary{}{}{}%
\setbox\matricesbox=\hbox{%
{$\left[\!\llap{\phantom{%
\begingroup \smaller\smaller\smaller\begin{tabular}{@{}c@{}}%
\phantom{0}\\\phantom{0}\\\phantom{0}
\end{tabular}\endgroup%
}}\right.$}%
\begingroup \smaller\smaller\smaller\begin{tabular}{@{}c@{}}%
-1/8\\\phantom{0}\\\phantom{0}
\end{tabular}\endgroup%
\kern3pt%
\begingroup \smaller\smaller\smaller\begin{tabular}{@{}c@{}}%
\phantom{0}\\5/2\\-1/2
\end{tabular}\endgroup%
\kern3pt%
\begingroup \smaller\smaller\smaller\begin{tabular}{@{}c@{}}%
\phantom{0}\\-1/2\\21/2
\end{tabular}\endgroup%
{$\left.\llap{\phantom{%
\begingroup \smaller\smaller\smaller\begin{tabular}{@{}c@{}}%
\phantom{0}\\\phantom{0}\\\phantom{0}
\end{tabular}\endgroup%
}}\!\right]$}%
{$\left[\!\llap{\phantom{%
\begingroup \smaller\smaller\smaller\begin{tabular}{@{}c@{}}%
0\\0\\0
\end{tabular}\endgroup%
}}\right.$}%
\begingroup \smaller\smaller\smaller\begin{tabular}{@{}c@{}}%
16\\-2\\-2
\end{tabular}\endgroup%
\kern3pt%
\begingroup \smaller\smaller\smaller\begin{tabular}{@{}c@{}}%
26\\2\\-3
\end{tabular}\endgroup%
\kern3pt%
\begingroup \smaller\smaller\smaller\begin{tabular}{@{}c@{}}%
2\\1\\0
\end{tabular}\endgroup%
{$\left.\llap{\phantom{%
\begingroup \smaller\smaller\smaller\begin{tabular}{@{}c@{}}%
0\\0\\0
\end{tabular}\endgroup%
}}\!\right]$}%
}%
\ifdim\wd\matricesbox>\halfwidth\myboxwidth=\hsize\else\myboxwidth=\halfwidth\fi
\vbox{%
\ifdim\myboxwidth=\hsize
\setbox\onelinebox=\hbox{%
\vbox{\hbox{%
$\Pi_{6,50}$ spans $L_{14.1}$%
}\hbox{%
$222222\rtimes C_{2}$%
}%
}%
\hfill\copy\matricesbox
}%
\ifdim\wd\onelinebox>\myboxwidth
\hbox to \myboxwidth{%
$\Pi_{6,50}$ spans $L_{14.1}$%
\hfil
$222222\rtimes C_{2}$%
}%
\box\matricesbox
\else
\hbox to \myboxwidth{%
\unhbox\onelinebox
}%
\fi
\else
\hbox to \myboxwidth{%
$\Pi_{6,50}$ spans $L_{14.1}$%
\hfil}%
\hbox to \myboxwidth{%
$222222\rtimes C_{2}$%
\hfil}%
\box\matricesbox
\fi
}%
\hfill\discretionary{}{}{}%
\setbox\matricesbox=\hbox{%
{$\left[\!\llap{\phantom{%
\begingroup \smaller\smaller\smaller\begin{tabular}{@{}c@{}}%
\phantom{0}\\\phantom{0}\\\phantom{0}
\end{tabular}\endgroup%
}}\right.$}%
\begingroup \smaller\smaller\smaller\begin{tabular}{@{}c@{}}%
-1/8\\\phantom{0}\\\phantom{0}
\end{tabular}\endgroup%
\kern3pt%
\begingroup \smaller\smaller\smaller\begin{tabular}{@{}c@{}}%
\phantom{0}\\21/2\\-3
\end{tabular}\endgroup%
\kern3pt%
\begingroup \smaller\smaller\smaller\begin{tabular}{@{}c@{}}%
\phantom{0}\\-3\\30
\end{tabular}\endgroup%
{$\left.\llap{\phantom{%
\begingroup \smaller\smaller\smaller\begin{tabular}{@{}c@{}}%
\phantom{0}\\\phantom{0}\\\phantom{0}
\end{tabular}\endgroup%
}}\!\right]$}%
{$\left[\!\llap{\phantom{%
\begingroup \smaller\smaller\smaller\begin{tabular}{@{}c@{}}%
0\\0\\0
\end{tabular}\endgroup%
}}\right.$}%
\begingroup \smaller\smaller\smaller\begin{tabular}{@{}c@{}}%
6\\1\\0
\end{tabular}\endgroup%
\kern3pt%
\begingroup \smaller\smaller\smaller\begin{tabular}{@{}c@{}}%
34\\3\\2
\end{tabular}\endgroup%
\kern3pt%
\begingroup \smaller\smaller\smaller\begin{tabular}{@{}c@{}}%
12\\0\\1
\end{tabular}\endgroup%
{$\left.\llap{\phantom{%
\begingroup \smaller\smaller\smaller\begin{tabular}{@{}c@{}}%
0\\0\\0
\end{tabular}\endgroup%
}}\!\right]$}%
}%
\ifdim\wd\matricesbox>\halfwidth\myboxwidth=\hsize\else\myboxwidth=\halfwidth\fi
\vbox{%
\ifdim\myboxwidth=\hsize
\setbox\onelinebox=\hbox{%
\vbox{\hbox{%
$\Pi_{6,51}$ spans $L_{48.4}$%
}\hbox{%
$224224\rtimes C_{2}$%
}%
}%
\hfill\copy\matricesbox
}%
\ifdim\wd\onelinebox>\myboxwidth
\hbox to \myboxwidth{%
$\Pi_{6,51}$ spans $L_{48.4}$%
\hfil
$224224\rtimes C_{2}$%
}%
\box\matricesbox
\else
\hbox to \myboxwidth{%
\unhbox\onelinebox
}%
\fi
\else
\hbox to \myboxwidth{%
$\Pi_{6,51}$ spans $L_{48.4}$%
\hfil}%
\hbox to \myboxwidth{%
$224224\rtimes C_{2}$%
\hfil}%
\box\matricesbox
\fi
}%
\hfill\discretionary{}{}{}%
\setbox\matricesbox=\hbox{%
{$\left[\!\llap{\phantom{%
\begingroup \smaller\smaller\smaller\begin{tabular}{@{}c@{}}%
\phantom{0}\\\phantom{0}\\\phantom{0}
\end{tabular}\endgroup%
}}\right.$}%
\begingroup \smaller\smaller\smaller\begin{tabular}{@{}c@{}}%
-1/8\\\phantom{0}\\\phantom{0}
\end{tabular}\endgroup%
\kern3pt%
\begingroup \smaller\smaller\smaller\begin{tabular}{@{}c@{}}%
\phantom{0}\\21/2\\-9/2
\end{tabular}\endgroup%
\kern3pt%
\begingroup \smaller\smaller\smaller\begin{tabular}{@{}c@{}}%
\phantom{0}\\-9/2\\117/2
\end{tabular}\endgroup%
{$\left.\llap{\phantom{%
\begingroup \smaller\smaller\smaller\begin{tabular}{@{}c@{}}%
\phantom{0}\\\phantom{0}\\\phantom{0}
\end{tabular}\endgroup%
}}\!\right]$}%
{$\left[\!\llap{\phantom{%
\begingroup \smaller\smaller\smaller\begin{tabular}{@{}c@{}}%
0\\0\\0
\end{tabular}\endgroup%
}}\right.$}%
\begingroup \smaller\smaller\smaller\begin{tabular}{@{}c@{}}%
6\\1\\0
\end{tabular}\endgroup%
\kern3pt%
\begingroup \smaller\smaller\smaller\begin{tabular}{@{}c@{}}%
22\\2\\1
\end{tabular}\endgroup%
\kern3pt%
\begingroup \smaller\smaller\smaller\begin{tabular}{@{}c@{}}%
18\\0\\1
\end{tabular}\endgroup%
{$\left.\llap{\phantom{%
\begingroup \smaller\smaller\smaller\begin{tabular}{@{}c@{}}%
0\\0\\0
\end{tabular}\endgroup%
}}\!\right]$}%
}%
\ifdim\wd\matricesbox>\halfwidth\myboxwidth=\hsize\else\myboxwidth=\halfwidth\fi
\vbox{%
\ifdim\myboxwidth=\hsize
\setbox\onelinebox=\hbox{%
\vbox{\hbox{%
$\Pi_{6,52}$ spans $L_{32.1}$%
}\hbox{%
$226226\rtimes C_{2}$%
}%
}%
\hfill\copy\matricesbox
}%
\ifdim\wd\onelinebox>\myboxwidth
\hbox to \myboxwidth{%
$\Pi_{6,52}$ spans $L_{32.1}$%
\hfil
$226226\rtimes C_{2}$%
}%
\box\matricesbox
\else
\hbox to \myboxwidth{%
\unhbox\onelinebox
}%
\fi
\else
\hbox to \myboxwidth{%
$\Pi_{6,52}$ spans $L_{32.1}$%
\hfil}%
\hbox to \myboxwidth{%
$226226\rtimes C_{2}$%
\hfil}%
\box\matricesbox
\fi
}%
\hfill\discretionary{}{}{}%
\setbox\matricesbox=\hbox{%
{$\left[\!\llap{\phantom{%
\begingroup \smaller\smaller\smaller\begin{tabular}{@{}c@{}}%
\phantom{0}\\\phantom{0}\\\phantom{0}
\end{tabular}\endgroup%
}}\right.$}%
\begingroup \smaller\smaller\smaller\begin{tabular}{@{}c@{}}%
-3/8\\\phantom{0}\\\phantom{0}
\end{tabular}\endgroup%
\kern3pt%
\begingroup \smaller\smaller\smaller\begin{tabular}{@{}c@{}}%
\phantom{0}\\7/2\\-1/2
\end{tabular}\endgroup%
\kern3pt%
\begingroup \smaller\smaller\smaller\begin{tabular}{@{}c@{}}%
\phantom{0}\\-1/2\\7/2
\end{tabular}\endgroup%
{$\left.\llap{\phantom{%
\begingroup \smaller\smaller\smaller\begin{tabular}{@{}c@{}}%
\phantom{0}\\\phantom{0}\\\phantom{0}
\end{tabular}\endgroup%
}}\!\right]$}%
{$\left[\!\llap{\phantom{%
\begingroup \smaller\smaller\smaller\begin{tabular}{@{}c@{}}%
0\\0\\0
\end{tabular}\endgroup%
}}\right.$}%
\begingroup \smaller\smaller\smaller\begin{tabular}{@{}c@{}}%
2\\0\\1
\end{tabular}\endgroup%
\kern3pt%
\begingroup \smaller\smaller\smaller\begin{tabular}{@{}c@{}}%
6\\-2\\1
\end{tabular}\endgroup%
\kern3pt%
\begingroup \smaller\smaller\smaller\begin{tabular}{@{}c@{}}%
8\\-3\\-1
\end{tabular}\endgroup%
{$\left.\llap{\phantom{%
\begingroup \smaller\smaller\smaller\begin{tabular}{@{}c@{}}%
0\\0\\0
\end{tabular}\endgroup%
}}\!\right]$}%
}%
\ifdim\wd\matricesbox>\halfwidth\myboxwidth=\hsize\else\myboxwidth=\halfwidth\fi
\vbox{%
\ifdim\myboxwidth=\hsize
\setbox\onelinebox=\hbox{%
\vbox{\hbox{%
$\Pi_{6,53}$ spans $L_{7.9}$%
}\hbox{%
$22\infty22\infty\rtimes C_{2}$%
}%
}%
\hfill\copy\matricesbox
}%
\ifdim\wd\onelinebox>\myboxwidth
\hbox to \myboxwidth{%
$\Pi_{6,53}$ spans $L_{7.9}$%
\hfil
$22\infty22\infty\rtimes C_{2}$%
}%
\box\matricesbox
\else
\hbox to \myboxwidth{%
\unhbox\onelinebox
}%
\fi
\else
\hbox to \myboxwidth{%
$\Pi_{6,53}$ spans $L_{7.9}$%
\hfil}%
\hbox to \myboxwidth{%
$22\infty22\infty\rtimes C_{2}$%
\hfil}%
\box\matricesbox
\fi
}%
\hfill\discretionary{}{}{}%
\setbox\matricesbox=\hbox{%
{$\left[\!\llap{\phantom{%
\begingroup \smaller\smaller\smaller\begin{tabular}{@{}c@{}}%
\phantom{0}\\\phantom{0}\\\phantom{0}
\end{tabular}\endgroup%
}}\right.$}%
\begingroup \smaller\smaller\smaller\begin{tabular}{@{}c@{}}%
-1/8\\\phantom{0}\\\phantom{0}
\end{tabular}\endgroup%
\kern3pt%
\begingroup \smaller\smaller\smaller\begin{tabular}{@{}c@{}}%
\phantom{0}\\5/2\\\phantom{0}
\end{tabular}\endgroup%
\kern3pt%
\begingroup \smaller\smaller\smaller\begin{tabular}{@{}c@{}}%
\phantom{0}\\\phantom{0}\\20
\end{tabular}\endgroup%
{$\left.\llap{\phantom{%
\begingroup \smaller\smaller\smaller\begin{tabular}{@{}c@{}}%
\phantom{0}\\\phantom{0}\\\phantom{0}
\end{tabular}\endgroup%
}}\!\right]$}%
{$\left[\!\llap{\phantom{%
\begingroup \smaller\smaller\smaller\begin{tabular}{@{}c@{}}%
0\\0\\0
\end{tabular}\endgroup%
}}\right.$}%
\begingroup \smaller\smaller\smaller\begin{tabular}{@{}c@{}}%
2\\-1\\0
\end{tabular}\endgroup%
\kern3pt%
\begingroup \smaller\smaller\smaller\begin{tabular}{@{}c@{}}%
80\\-8\\-6
\end{tabular}\endgroup%
\kern3pt%
\begingroup \smaller\smaller\smaller\begin{tabular}{@{}c@{}}%
10\\1\\-1
\end{tabular}\endgroup%
{$\left.\llap{\phantom{%
\begingroup \smaller\smaller\smaller\begin{tabular}{@{}c@{}}%
0\\0\\0
\end{tabular}\endgroup%
}}\!\right]$}%
}%
\ifdim\wd\matricesbox>\halfwidth\myboxwidth=\hsize\else\myboxwidth=\halfwidth\fi
\vbox{%
\ifdim\myboxwidth=\hsize
\setbox\onelinebox=\hbox{%
\vbox{\hbox{%
$\Pi_{6,54}$ spans $L_{6.5}$%
}\hbox{%
$222222\rtimes C_{2}$%
}%
}%
\hfill\copy\matricesbox
}%
\ifdim\wd\onelinebox>\myboxwidth
\hbox to \myboxwidth{%
$\Pi_{6,54}$ spans $L_{6.5}$%
\hfil
$222222\rtimes C_{2}$%
}%
\box\matricesbox
\else
\hbox to \myboxwidth{%
\unhbox\onelinebox
}%
\fi
\else
\hbox to \myboxwidth{%
$\Pi_{6,54}$ spans $L_{6.5}$%
\hfil}%
\hbox to \myboxwidth{%
$222222\rtimes C_{2}$%
\hfil}%
\box\matricesbox
\fi
}%
\hfill\discretionary{}{}{}%
\setbox\matricesbox=\hbox{%
{$\left[\!\llap{\phantom{%
\begingroup \smaller\smaller\smaller\begin{tabular}{@{}c@{}}%
\phantom{0}\\\phantom{0}\\\phantom{0}
\end{tabular}\endgroup%
}}\right.$}%
\begingroup \smaller\smaller\smaller\begin{tabular}{@{}c@{}}%
-1/8\\\phantom{0}\\\phantom{0}
\end{tabular}\endgroup%
\kern3pt%
\begingroup \smaller\smaller\smaller\begin{tabular}{@{}c@{}}%
\phantom{0}\\6\\-3
\end{tabular}\endgroup%
\kern3pt%
\begingroup \smaller\smaller\smaller\begin{tabular}{@{}c@{}}%
\phantom{0}\\-3\\21/2
\end{tabular}\endgroup%
{$\left.\llap{\phantom{%
\begingroup \smaller\smaller\smaller\begin{tabular}{@{}c@{}}%
\phantom{0}\\\phantom{0}\\\phantom{0}
\end{tabular}\endgroup%
}}\!\right]$}%
{$\left[\!\llap{\phantom{%
\begingroup \smaller\smaller\smaller\begin{tabular}{@{}c@{}}%
0\\0\\0
\end{tabular}\endgroup%
}}\right.$}%
\begingroup \smaller\smaller\smaller\begin{tabular}{@{}c@{}}%
4\\1\\0
\end{tabular}\endgroup%
\kern3pt%
\begingroup \smaller\smaller\smaller\begin{tabular}{@{}c@{}}%
36\\1\\-4
\end{tabular}\endgroup%
\kern3pt%
\begingroup \smaller\smaller\smaller\begin{tabular}{@{}c@{}}%
6\\-1\\-1
\end{tabular}\endgroup%
{$\left.\llap{\phantom{%
\begingroup \smaller\smaller\smaller\begin{tabular}{@{}c@{}}%
0\\0\\0
\end{tabular}\endgroup%
}}\!\right]$}%
}%
\ifdim\wd\matricesbox>\halfwidth\myboxwidth=\hsize\else\myboxwidth=\halfwidth\fi
\vbox{%
\ifdim\myboxwidth=\hsize
\setbox\onelinebox=\hbox{%
\vbox{\hbox{%
$\Pi_{6,55}$ spans $L_{154.1}$%
}\hbox{%
$222222\rtimes C_{2}$%
}%
}%
\hfill\copy\matricesbox
}%
\ifdim\wd\onelinebox>\myboxwidth
\hbox to \myboxwidth{%
$\Pi_{6,55}$ spans $L_{154.1}$%
\hfil
$222222\rtimes C_{2}$%
}%
\box\matricesbox
\else
\hbox to \myboxwidth{%
\unhbox\onelinebox
}%
\fi
\else
\hbox to \myboxwidth{%
$\Pi_{6,55}$ spans $L_{154.1}$%
\hfil}%
\hbox to \myboxwidth{%
$222222\rtimes C_{2}$%
\hfil}%
\box\matricesbox
\fi
}%
\hfill\discretionary{}{}{}%
\setbox\matricesbox=\hbox{%
{$\left[\!\llap{\phantom{%
\begingroup \smaller\smaller\smaller\begin{tabular}{@{}c@{}}%
\phantom{0}\\\phantom{0}\\\phantom{0}
\end{tabular}\endgroup%
}}\right.$}%
\begingroup \smaller\smaller\smaller\begin{tabular}{@{}c@{}}%
-1/8\\\phantom{0}\\\phantom{0}
\end{tabular}\endgroup%
\kern3pt%
\begingroup \smaller\smaller\smaller\begin{tabular}{@{}c@{}}%
\phantom{0}\\21/2\\-3
\end{tabular}\endgroup%
\kern3pt%
\begingroup \smaller\smaller\smaller\begin{tabular}{@{}c@{}}%
\phantom{0}\\-3\\18
\end{tabular}\endgroup%
{$\left.\llap{\phantom{%
\begingroup \smaller\smaller\smaller\begin{tabular}{@{}c@{}}%
\phantom{0}\\\phantom{0}\\\phantom{0}
\end{tabular}\endgroup%
}}\!\right]$}%
{$\left[\!\llap{\phantom{%
\begingroup \smaller\smaller\smaller\begin{tabular}{@{}c@{}}%
0\\0\\0
\end{tabular}\endgroup%
}}\right.$}%
\begingroup \smaller\smaller\smaller\begin{tabular}{@{}c@{}}%
10\\-1\\-1
\end{tabular}\endgroup%
\kern3pt%
\begingroup \smaller\smaller\smaller\begin{tabular}{@{}c@{}}%
40\\2\\-3
\end{tabular}\endgroup%
\kern3pt%
\begingroup \smaller\smaller\smaller\begin{tabular}{@{}c@{}}%
6\\1\\0
\end{tabular}\endgroup%
{$\left.\llap{\phantom{%
\begingroup \smaller\smaller\smaller\begin{tabular}{@{}c@{}}%
0\\0\\0
\end{tabular}\endgroup%
}}\!\right]$}%
}%
\ifdim\wd\matricesbox>\halfwidth\myboxwidth=\hsize\else\myboxwidth=\halfwidth\fi
\vbox{%
\ifdim\myboxwidth=\hsize
\setbox\onelinebox=\hbox{%
\vbox{\hbox{%
$\Pi_{6,56}$ spans $L_{30.15}$%
}\hbox{%
$\infty22\infty22\rtimes C_{2}$%
}%
}%
\hfill\copy\matricesbox
}%
\ifdim\wd\onelinebox>\myboxwidth
\hbox to \myboxwidth{%
$\Pi_{6,56}$ spans $L_{30.15}$%
\hfil
$\infty22\infty22\rtimes C_{2}$%
}%
\box\matricesbox
\else
\hbox to \myboxwidth{%
\unhbox\onelinebox
}%
\fi
\else
\hbox to \myboxwidth{%
$\Pi_{6,56}$ spans $L_{30.15}$%
\hfil}%
\hbox to \myboxwidth{%
$\infty22\infty22\rtimes C_{2}$%
\hfil}%
\box\matricesbox
\fi
}%
\hfill\discretionary{}{}{}%
\setbox\matricesbox=\hbox{%
{$\left[\!\llap{\phantom{%
\begingroup \smaller\smaller\smaller\begin{tabular}{@{}c@{}}%
\phantom{0}\\\phantom{0}\\\phantom{0}
\end{tabular}\endgroup%
}}\right.$}%
\begingroup \smaller\smaller\smaller\begin{tabular}{@{}c@{}}%
-1/8\\\phantom{0}\\\phantom{0}
\end{tabular}\endgroup%
\kern3pt%
\begingroup \smaller\smaller\smaller\begin{tabular}{@{}c@{}}%
\phantom{0}\\5/2\\-1
\end{tabular}\endgroup%
\kern3pt%
\begingroup \smaller\smaller\smaller\begin{tabular}{@{}c@{}}%
\phantom{0}\\-1\\18
\end{tabular}\endgroup%
{$\left.\llap{\phantom{%
\begingroup \smaller\smaller\smaller\begin{tabular}{@{}c@{}}%
\phantom{0}\\\phantom{0}\\\phantom{0}
\end{tabular}\endgroup%
}}\!\right]$}%
{$\left[\!\llap{\phantom{%
\begingroup \smaller\smaller\smaller\begin{tabular}{@{}c@{}}%
0\\0\\0
\end{tabular}\endgroup%
}}\right.$}%
\begingroup \smaller\smaller\smaller\begin{tabular}{@{}c@{}}%
22\\-3\\-2
\end{tabular}\endgroup%
\kern3pt%
\begingroup \smaller\smaller\smaller\begin{tabular}{@{}c@{}}%
88\\6\\-7
\end{tabular}\endgroup%
\kern3pt%
\begingroup \smaller\smaller\smaller\begin{tabular}{@{}c@{}}%
2\\1\\0
\end{tabular}\endgroup%
{$\left.\llap{\phantom{%
\begingroup \smaller\smaller\smaller\begin{tabular}{@{}c@{}}%
0\\0\\0
\end{tabular}\endgroup%
}}\!\right]$}%
}%
\ifdim\wd\matricesbox>\halfwidth\myboxwidth=\hsize\else\myboxwidth=\halfwidth\fi
\vbox{%
\ifdim\myboxwidth=\hsize
\setbox\onelinebox=\hbox{%
\vbox{\hbox{%
$\Pi_{6,57}$ spans $L_{26.1}$%
}\hbox{%
$\infty22\infty22\rtimes C_{2}$%
}%
}%
\hfill\copy\matricesbox
}%
\ifdim\wd\onelinebox>\myboxwidth
\hbox to \myboxwidth{%
$\Pi_{6,57}$ spans $L_{26.1}$%
\hfil
$\infty22\infty22\rtimes C_{2}$%
}%
\box\matricesbox
\else
\hbox to \myboxwidth{%
\unhbox\onelinebox
}%
\fi
\else
\hbox to \myboxwidth{%
$\Pi_{6,57}$ spans $L_{26.1}$%
\hfil}%
\hbox to \myboxwidth{%
$\infty22\infty22\rtimes C_{2}$%
\hfil}%
\box\matricesbox
\fi
}%
\hfill\discretionary{}{}{}%
\setbox\matricesbox=\hbox{%
{$\left[\!\llap{\phantom{%
\begingroup \smaller\smaller\smaller\begin{tabular}{@{}c@{}}%
\phantom{0}\\\phantom{0}\\\phantom{0}
\end{tabular}\endgroup%
}}\right.$}%
\begingroup \smaller\smaller\smaller\begin{tabular}{@{}c@{}}%
-1/11\\\phantom{0}\\\phantom{0}
\end{tabular}\endgroup%
\kern3pt%
\begingroup \smaller\smaller\smaller\begin{tabular}{@{}c@{}}%
\phantom{0}\\42/11\\-6/11
\end{tabular}\endgroup%
\kern3pt%
\begingroup \smaller\smaller\smaller\begin{tabular}{@{}c@{}}%
\phantom{0}\\-6/11\\48/11
\end{tabular}\endgroup%
{$\left.\llap{\phantom{%
\begingroup \smaller\smaller\smaller\begin{tabular}{@{}c@{}}%
\phantom{0}\\\phantom{0}\\\phantom{0}
\end{tabular}\endgroup%
}}\!\right]$}%
{$\left[\!\llap{\phantom{%
\begingroup \smaller\smaller\smaller\begin{tabular}{@{}c@{}}%
0\\0\\0
\end{tabular}\endgroup%
}}\right.$}%
\begingroup \smaller\smaller\smaller\begin{tabular}{@{}c@{}}%
3\\1\\0
\end{tabular}\endgroup%
\kern3pt%
\begingroup \smaller\smaller\smaller\begin{tabular}{@{}c@{}}%
40\\2\\-6
\end{tabular}\endgroup%
\kern3pt%
\begingroup \smaller\smaller\smaller\begin{tabular}{@{}c@{}}%
24\\-2\\-4
\end{tabular}\endgroup%
\kern3pt%
\begingroup \smaller\smaller\smaller\begin{tabular}{@{}c@{}}%
15\\-3\\-1
\end{tabular}\endgroup%
\kern3pt%
\begingroup \smaller\smaller\smaller\begin{tabular}{@{}c@{}}%
6\\-1\\1
\end{tabular}\endgroup%
\kern3pt%
\begingroup \smaller\smaller\smaller\begin{tabular}{@{}c@{}}%
10\\1\\2
\end{tabular}\endgroup%
{$\left.\llap{\phantom{%
\begingroup \smaller\smaller\smaller\begin{tabular}{@{}c@{}}%
0\\0\\0
\end{tabular}\endgroup%
}}\!\right]$}%
}%
\ifdim\wd\matricesbox>\halfwidth\myboxwidth=\hsize\else\myboxwidth=\halfwidth\fi
\vbox{%
\ifdim\myboxwidth=\hsize
\setbox\onelinebox=\hbox{%
\vbox{\hbox{%
$\Pi_{6,58}$ spans $L_{17.14}$%
}\hbox{%
$222222$%
}%
}%
\hfill\copy\matricesbox
}%
\ifdim\wd\onelinebox>\myboxwidth
\hbox to \myboxwidth{%
$\Pi_{6,58}$ spans $L_{17.14}$%
\hfil
$222222$%
}%
\box\matricesbox
\else
\hbox to \myboxwidth{%
\unhbox\onelinebox
}%
\fi
\else
\hbox to \myboxwidth{%
$\Pi_{6,58}$ spans $L_{17.14}$%
\hfil}%
\hbox to \myboxwidth{%
$222222$%
\hfil}%
\box\matricesbox
\fi
}%
\hfill\discretionary{}{}{}%

\vskip2pt\hrule\vskip2pt

\leavevmode\setbox\matricesbox=\hbox{%
{$\left[\!\llap{\phantom{%
\begingroup \smaller\smaller\smaller\begin{tabular}{@{}c@{}}%
\phantom{0}\\\phantom{0}\\\phantom{0}
\end{tabular}\endgroup%
}}\right.$}%
\begingroup \smaller\smaller\smaller\begin{tabular}{@{}c@{}}%
-1/2\\\phantom{0}\\\phantom{0}
\end{tabular}\endgroup%
\kern3pt%
\begingroup \smaller\smaller\smaller\begin{tabular}{@{}c@{}}%
\phantom{0}\\6\\\phantom{0}
\end{tabular}\endgroup%
\kern3pt%
\begingroup \smaller\smaller\smaller\begin{tabular}{@{}c@{}}%
\phantom{0}\\\phantom{0}\\3/2
\end{tabular}\endgroup%
{$\left.\llap{\phantom{%
\begingroup \smaller\smaller\smaller\begin{tabular}{@{}c@{}}%
\phantom{0}\\\phantom{0}\\\phantom{0}
\end{tabular}\endgroup%
}}\!\right]$}%
{$\left[\!\llap{\phantom{%
\begingroup \smaller\smaller\smaller\begin{tabular}{@{}c@{}}%
0\\0\\0
\end{tabular}\endgroup%
}}\right.$}%
\begingroup \smaller\smaller\smaller\begin{tabular}{@{}c@{}}%
6\\-2\\0
\end{tabular}\endgroup%
\kern3pt%
\begingroup \smaller\smaller\smaller\begin{tabular}{@{}c@{}}%
24\\-6\\-8
\end{tabular}\endgroup%
\kern3pt%
\begingroup \smaller\smaller\smaller\begin{tabular}{@{}c@{}}%
1\\0\\-1
\end{tabular}\endgroup%
\kern3pt%
\begingroup \smaller\smaller\smaller\begin{tabular}{@{}c@{}}%
3\\1\\-1
\end{tabular}\endgroup%
{$\left.\llap{\phantom{%
\begingroup \smaller\smaller\smaller\begin{tabular}{@{}c@{}}%
0\\0\\0
\end{tabular}\endgroup%
}}\!\right]$}%
}%
\ifdim\wd\matricesbox>\halfwidth\myboxwidth=\hsize\else\myboxwidth=\halfwidth\fi
\vbox{%
\ifdim\myboxwidth=\hsize
\setbox\onelinebox=\hbox{%
\vbox{\hbox{%
$\Pi_{7,1}$ spans $L_{123.6}$%
}\hbox{%
$22|222\slashtwo2\rtimes D_{2}$%
}%
}%
\hfill\copy\matricesbox
}%
\ifdim\wd\onelinebox>\myboxwidth
\hbox to \myboxwidth{%
$\Pi_{7,1}$ spans $L_{123.6}$%
\hfil
$22|222\slashtwo2\rtimes D_{2}$%
}%
\box\matricesbox
\else
\hbox to \myboxwidth{%
\unhbox\onelinebox
}%
\fi
\else
\hbox to \myboxwidth{%
$\Pi_{7,1}$ spans $L_{123.6}$%
\hfil}%
\hbox to \myboxwidth{%
$22|222\slashtwo2\rtimes D_{2}$%
\hfil}%
\box\matricesbox
\fi
}%
\hfill\discretionary{}{}{}%
\setbox\matricesbox=\hbox{%
{$\left[\!\llap{\phantom{%
\begingroup \smaller\smaller\smaller\begin{tabular}{@{}c@{}}%
\phantom{0}\\\phantom{0}\\\phantom{0}
\end{tabular}\endgroup%
}}\right.$}%
\begingroup \smaller\smaller\smaller\begin{tabular}{@{}c@{}}%
-1\\\phantom{0}\\\phantom{0}
\end{tabular}\endgroup%
\kern3pt%
\begingroup \smaller\smaller\smaller\begin{tabular}{@{}c@{}}%
\phantom{0}\\3/2\\\phantom{0}
\end{tabular}\endgroup%
\kern3pt%
\begingroup \smaller\smaller\smaller\begin{tabular}{@{}c@{}}%
\phantom{0}\\\phantom{0}\\1/2
\end{tabular}\endgroup%
{$\left.\llap{\phantom{%
\begingroup \smaller\smaller\smaller\begin{tabular}{@{}c@{}}%
\phantom{0}\\\phantom{0}\\\phantom{0}
\end{tabular}\endgroup%
}}\!\right]$}%
{$\left[\!\llap{\phantom{%
\begingroup \smaller\smaller\smaller\begin{tabular}{@{}c@{}}%
0\\0\\0
\end{tabular}\endgroup%
}}\right.$}%
\begingroup \smaller\smaller\smaller\begin{tabular}{@{}c@{}}%
2\\2\\0
\end{tabular}\endgroup%
\kern3pt%
\begingroup \smaller\smaller\smaller\begin{tabular}{@{}c@{}}%
6\\4\\6
\end{tabular}\endgroup%
\kern3pt%
\begingroup \smaller\smaller\smaller\begin{tabular}{@{}c@{}}%
1\\0\\2
\end{tabular}\endgroup%
\kern3pt%
\begingroup \smaller\smaller\smaller\begin{tabular}{@{}c@{}}%
1\\-1\\1
\end{tabular}\endgroup%
{$\left.\llap{\phantom{%
\begingroup \smaller\smaller\smaller\begin{tabular}{@{}c@{}}%
0\\0\\0
\end{tabular}\endgroup%
}}\!\right]$}%
}%
\ifdim\wd\matricesbox>\halfwidth\myboxwidth=\hsize\else\myboxwidth=\halfwidth\fi
\vbox{%
\ifdim\myboxwidth=\hsize
\setbox\onelinebox=\hbox{%
\vbox{\hbox{%
$\Pi_{7,2}$ spans $L_{3.2}$%
}\hbox{%
$22|222\slashtwo2\rtimes D_{2}$%
}%
}%
\hfill\copy\matricesbox
}%
\ifdim\wd\onelinebox>\myboxwidth
\hbox to \myboxwidth{%
$\Pi_{7,2}$ spans $L_{3.2}$%
\hfil
$22|222\slashtwo2\rtimes D_{2}$%
}%
\box\matricesbox
\else
\hbox to \myboxwidth{%
\unhbox\onelinebox
}%
\fi
\else
\hbox to \myboxwidth{%
$\Pi_{7,2}$ spans $L_{3.2}$%
\hfil}%
\hbox to \myboxwidth{%
$22|222\slashtwo2\rtimes D_{2}$%
\hfil}%
\box\matricesbox
\fi
}%
\hfill\discretionary{}{}{}%
\setbox\matricesbox=\hbox{%
{$\left[\!\llap{\phantom{%
\begingroup \smaller\smaller\smaller\begin{tabular}{@{}c@{}}%
\phantom{0}\\\phantom{0}\\\phantom{0}
\end{tabular}\endgroup%
}}\right.$}%
\begingroup \smaller\smaller\smaller\begin{tabular}{@{}c@{}}%
-1/4\\\phantom{0}\\\phantom{0}
\end{tabular}\endgroup%
\kern3pt%
\begingroup \smaller\smaller\smaller\begin{tabular}{@{}c@{}}%
\phantom{0}\\15/4\\\phantom{0}
\end{tabular}\endgroup%
\kern3pt%
\begingroup \smaller\smaller\smaller\begin{tabular}{@{}c@{}}%
\phantom{0}\\\phantom{0}\\3/2
\end{tabular}\endgroup%
{$\left.\llap{\phantom{%
\begingroup \smaller\smaller\smaller\begin{tabular}{@{}c@{}}%
\phantom{0}\\\phantom{0}\\\phantom{0}
\end{tabular}\endgroup%
}}\!\right]$}%
{$\left[\!\llap{\phantom{%
\begingroup \smaller\smaller\smaller\begin{tabular}{@{}c@{}}%
0\\0\\0
\end{tabular}\endgroup%
}}\right.$}%
\begingroup \smaller\smaller\smaller\begin{tabular}{@{}c@{}}%
6\\2\\0
\end{tabular}\endgroup%
\kern3pt%
\begingroup \smaller\smaller\smaller\begin{tabular}{@{}c@{}}%
15\\3\\5
\end{tabular}\endgroup%
\kern3pt%
\begingroup \smaller\smaller\smaller\begin{tabular}{@{}c@{}}%
8\\0\\4
\end{tabular}\endgroup%
\kern3pt%
\begingroup \smaller\smaller\smaller\begin{tabular}{@{}c@{}}%
3\\-1\\1
\end{tabular}\endgroup%
{$\left.\llap{\phantom{%
\begingroup \smaller\smaller\smaller\begin{tabular}{@{}c@{}}%
0\\0\\0
\end{tabular}\endgroup%
}}\!\right]$}%
}%
\ifdim\wd\matricesbox>\halfwidth\myboxwidth=\hsize\else\myboxwidth=\halfwidth\fi
\vbox{%
\ifdim\myboxwidth=\hsize
\setbox\onelinebox=\hbox{%
\vbox{\hbox{%
$\Pi_{7,3}$ spans $L_{127.6}$%
}\hbox{%
$|222\slashtwo222\rtimes D_{2}$%
}%
}%
\hfill\copy\matricesbox
}%
\ifdim\wd\onelinebox>\myboxwidth
\hbox to \myboxwidth{%
$\Pi_{7,3}$ spans $L_{127.6}$%
\hfil
$|222\slashtwo222\rtimes D_{2}$%
}%
\box\matricesbox
\else
\hbox to \myboxwidth{%
\unhbox\onelinebox
}%
\fi
\else
\hbox to \myboxwidth{%
$\Pi_{7,3}$ spans $L_{127.6}$%
\hfil}%
\hbox to \myboxwidth{%
$|222\slashtwo222\rtimes D_{2}$%
\hfil}%
\box\matricesbox
\fi
}%
\hfill\discretionary{}{}{}%
\setbox\matricesbox=\hbox{%
{$\left[\!\llap{\phantom{%
\begingroup \smaller\smaller\smaller\begin{tabular}{@{}c@{}}%
\phantom{0}\\\phantom{0}\\\phantom{0}
\end{tabular}\endgroup%
}}\right.$}%
\begingroup \smaller\smaller\smaller\begin{tabular}{@{}c@{}}%
-1/5\\\phantom{0}\\\phantom{0}
\end{tabular}\endgroup%
\kern3pt%
\begingroup \smaller\smaller\smaller\begin{tabular}{@{}c@{}}%
\phantom{0}\\21/5\\\phantom{0}
\end{tabular}\endgroup%
\kern3pt%
\begingroup \smaller\smaller\smaller\begin{tabular}{@{}c@{}}%
\phantom{0}\\\phantom{0}\\1
\end{tabular}\endgroup%
{$\left.\llap{\phantom{%
\begingroup \smaller\smaller\smaller\begin{tabular}{@{}c@{}}%
\phantom{0}\\\phantom{0}\\\phantom{0}
\end{tabular}\endgroup%
}}\!\right]$}%
{$\left[\!\llap{\phantom{%
\begingroup \smaller\smaller\smaller\begin{tabular}{@{}c@{}}%
0\\0\\0
\end{tabular}\endgroup%
}}\right.$}%
\begingroup \smaller\smaller\smaller\begin{tabular}{@{}c@{}}%
8\\2\\-2
\end{tabular}\endgroup%
\kern3pt%
\begingroup \smaller\smaller\smaller\begin{tabular}{@{}c@{}}%
14\\1\\-7
\end{tabular}\endgroup%
\kern3pt%
\begingroup \smaller\smaller\smaller\begin{tabular}{@{}c@{}}%
6\\-1\\-3
\end{tabular}\endgroup%
\kern3pt%
\begingroup \smaller\smaller\smaller\begin{tabular}{@{}c@{}}%
7\\-2\\0
\end{tabular}\endgroup%
{$\left.\llap{\phantom{%
\begingroup \smaller\smaller\smaller\begin{tabular}{@{}c@{}}%
0\\0\\0
\end{tabular}\endgroup%
}}\!\right]$}%
}%
\ifdim\wd\matricesbox>\halfwidth\myboxwidth=\hsize\else\myboxwidth=\halfwidth\fi
\vbox{%
\ifdim\myboxwidth=\hsize
\setbox\onelinebox=\hbox{%
\vbox{\hbox{%
$\Pi_{7,4}$ spans $L_{24.8}$%
}\hbox{%
$22\slashtwo222|2\rtimes D_{2}$%
}%
}%
\hfill\copy\matricesbox
}%
\ifdim\wd\onelinebox>\myboxwidth
\hbox to \myboxwidth{%
$\Pi_{7,4}$ spans $L_{24.8}$%
\hfil
$22\slashtwo222|2\rtimes D_{2}$%
}%
\box\matricesbox
\else
\hbox to \myboxwidth{%
\unhbox\onelinebox
}%
\fi
\else
\hbox to \myboxwidth{%
$\Pi_{7,4}$ spans $L_{24.8}$%
\hfil}%
\hbox to \myboxwidth{%
$22\slashtwo222|2\rtimes D_{2}$%
\hfil}%
\box\matricesbox
\fi
}%
\hfill\discretionary{}{}{}%
\setbox\matricesbox=\hbox{%
{$\left[\!\llap{\phantom{%
\begingroup \smaller\smaller\smaller\begin{tabular}{@{}c@{}}%
\phantom{0}\\\phantom{0}\\\phantom{0}
\end{tabular}\endgroup%
}}\right.$}%
\begingroup \smaller\smaller\smaller\begin{tabular}{@{}c@{}}%
-1\\\phantom{0}\\\phantom{0}
\end{tabular}\endgroup%
\kern3pt%
\begingroup \smaller\smaller\smaller\begin{tabular}{@{}c@{}}%
\phantom{0}\\1/2\\\phantom{0}
\end{tabular}\endgroup%
\kern3pt%
\begingroup \smaller\smaller\smaller\begin{tabular}{@{}c@{}}%
\phantom{0}\\\phantom{0}\\3/2
\end{tabular}\endgroup%
{$\left.\llap{\phantom{%
\begingroup \smaller\smaller\smaller\begin{tabular}{@{}c@{}}%
\phantom{0}\\\phantom{0}\\\phantom{0}
\end{tabular}\endgroup%
}}\!\right]$}%
{$\left[\!\llap{\phantom{%
\begingroup \smaller\smaller\smaller\begin{tabular}{@{}c@{}}%
0\\0\\0
\end{tabular}\endgroup%
}}\right.$}%
\begingroup \smaller\smaller\smaller\begin{tabular}{@{}c@{}}%
1\\-2\\0
\end{tabular}\endgroup%
\kern3pt%
\begingroup \smaller\smaller\smaller\begin{tabular}{@{}c@{}}%
6\\-6\\-4
\end{tabular}\endgroup%
\kern3pt%
\begingroup \smaller\smaller\smaller\begin{tabular}{@{}c@{}}%
2\\0\\-2
\end{tabular}\endgroup%
\kern3pt%
\begingroup \smaller\smaller\smaller\begin{tabular}{@{}c@{}}%
2\\3\\-1
\end{tabular}\endgroup%
{$\left.\llap{\phantom{%
\begingroup \smaller\smaller\smaller\begin{tabular}{@{}c@{}}%
0\\0\\0
\end{tabular}\endgroup%
}}\!\right]$}%
}%
\ifdim\wd\matricesbox>\halfwidth\myboxwidth=\hsize\else\myboxwidth=\halfwidth\fi
\vbox{%
\ifdim\myboxwidth=\hsize
\setbox\onelinebox=\hbox{%
\vbox{\hbox{%
$\Pi_{7,5}$ spans $L_{3.2}$%
}\hbox{%
$22|223\slashthree3\rtimes D_{2}$%
}%
}%
\hfill\copy\matricesbox
}%
\ifdim\wd\onelinebox>\myboxwidth
\hbox to \myboxwidth{%
$\Pi_{7,5}$ spans $L_{3.2}$%
\hfil
$22|223\slashthree3\rtimes D_{2}$%
}%
\box\matricesbox
\else
\hbox to \myboxwidth{%
\unhbox\onelinebox
}%
\fi
\else
\hbox to \myboxwidth{%
$\Pi_{7,5}$ spans $L_{3.2}$%
\hfil}%
\hbox to \myboxwidth{%
$22|223\slashthree3\rtimes D_{2}$%
\hfil}%
\box\matricesbox
\fi
}%
\hfill\discretionary{}{}{}%
\setbox\matricesbox=\hbox{%
{$\left[\!\llap{\phantom{%
\begingroup \smaller\smaller\smaller\begin{tabular}{@{}c@{}}%
\phantom{0}\\\phantom{0}\\\phantom{0}
\end{tabular}\endgroup%
}}\right.$}%
\begingroup \smaller\smaller\smaller\begin{tabular}{@{}c@{}}%
-1\\\phantom{0}\\\phantom{0}
\end{tabular}\endgroup%
\kern3pt%
\begingroup \smaller\smaller\smaller\begin{tabular}{@{}c@{}}%
\phantom{0}\\1/2\\\phantom{0}
\end{tabular}\endgroup%
\kern3pt%
\begingroup \smaller\smaller\smaller\begin{tabular}{@{}c@{}}%
\phantom{0}\\\phantom{0}\\3/2
\end{tabular}\endgroup%
{$\left.\llap{\phantom{%
\begingroup \smaller\smaller\smaller\begin{tabular}{@{}c@{}}%
\phantom{0}\\\phantom{0}\\\phantom{0}
\end{tabular}\endgroup%
}}\!\right]$}%
{$\left[\!\llap{\phantom{%
\begingroup \smaller\smaller\smaller\begin{tabular}{@{}c@{}}%
0\\0\\0
\end{tabular}\endgroup%
}}\right.$}%
\begingroup \smaller\smaller\smaller\begin{tabular}{@{}c@{}}%
1\\2\\0
\end{tabular}\endgroup%
\kern3pt%
\begingroup \smaller\smaller\smaller\begin{tabular}{@{}c@{}}%
1\\1\\1
\end{tabular}\endgroup%
\kern3pt%
\begingroup \smaller\smaller\smaller\begin{tabular}{@{}c@{}}%
6\\-3\\5
\end{tabular}\endgroup%
\kern3pt%
\begingroup \smaller\smaller\smaller\begin{tabular}{@{}c@{}}%
2\\-3\\1
\end{tabular}\endgroup%
{$\left.\llap{\phantom{%
\begingroup \smaller\smaller\smaller\begin{tabular}{@{}c@{}}%
0\\0\\0
\end{tabular}\endgroup%
}}\!\right]$}%
}%
\ifdim\wd\matricesbox>\halfwidth\myboxwidth=\hsize\else\myboxwidth=\halfwidth\fi
\vbox{%
\ifdim\myboxwidth=\hsize
\setbox\onelinebox=\hbox{%
\vbox{\hbox{%
$\Pi_{7,6}$ spans $L_{3.2}$%
}\hbox{%
$222|222\slashthree\rtimes D_{2}$%
}%
}%
\hfill\copy\matricesbox
}%
\ifdim\wd\onelinebox>\myboxwidth
\hbox to \myboxwidth{%
$\Pi_{7,6}$ spans $L_{3.2}$%
\hfil
$222|222\slashthree\rtimes D_{2}$%
}%
\box\matricesbox
\else
\hbox to \myboxwidth{%
\unhbox\onelinebox
}%
\fi
\else
\hbox to \myboxwidth{%
$\Pi_{7,6}$ spans $L_{3.2}$%
\hfil}%
\hbox to \myboxwidth{%
$222|222\slashthree\rtimes D_{2}$%
\hfil}%
\box\matricesbox
\fi
}%
\hfill\discretionary{}{}{}%
\setbox\matricesbox=\hbox{%
{$\left[\!\llap{\phantom{%
\begingroup \smaller\smaller\smaller\begin{tabular}{@{}c@{}}%
\phantom{0}\\\phantom{0}\\\phantom{0}
\end{tabular}\endgroup%
}}\right.$}%
\begingroup \smaller\smaller\smaller\begin{tabular}{@{}c@{}}%
-1/3\\\phantom{0}\\\phantom{0}
\end{tabular}\endgroup%
\kern3pt%
\begingroup \smaller\smaller\smaller\begin{tabular}{@{}c@{}}%
\phantom{0}\\1/3\\\phantom{0}
\end{tabular}\endgroup%
\kern3pt%
\begingroup \smaller\smaller\smaller\begin{tabular}{@{}c@{}}%
\phantom{0}\\\phantom{0}\\7
\end{tabular}\endgroup%
{$\left.\llap{\phantom{%
\begingroup \smaller\smaller\smaller\begin{tabular}{@{}c@{}}%
\phantom{0}\\\phantom{0}\\\phantom{0}
\end{tabular}\endgroup%
}}\!\right]$}%
{$\left[\!\llap{\phantom{%
\begingroup \smaller\smaller\smaller\begin{tabular}{@{}c@{}}%
0\\0\\0
\end{tabular}\endgroup%
}}\right.$}%
\begingroup \smaller\smaller\smaller\begin{tabular}{@{}c@{}}%
1\\2\\0
\end{tabular}\endgroup%
\kern3pt%
\begingroup \smaller\smaller\smaller\begin{tabular}{@{}c@{}}%
14\\7\\-3
\end{tabular}\endgroup%
\kern3pt%
\begingroup \smaller\smaller\smaller\begin{tabular}{@{}c@{}}%
8\\-2\\-2
\end{tabular}\endgroup%
\kern3pt%
\begingroup \smaller\smaller\smaller\begin{tabular}{@{}c@{}}%
7\\-7\\-1
\end{tabular}\endgroup%
{$\left.\llap{\phantom{%
\begingroup \smaller\smaller\smaller\begin{tabular}{@{}c@{}}%
0\\0\\0
\end{tabular}\endgroup%
}}\!\right]$}%
}%
\ifdim\wd\matricesbox>\halfwidth\myboxwidth=\hsize\else\myboxwidth=\halfwidth\fi
\vbox{%
\ifdim\myboxwidth=\hsize
\setbox\onelinebox=\hbox{%
\vbox{\hbox{%
$\Pi_{7,7}$ spans $L_{9.7}$%
}\hbox{%
$22|222\slashinfty2\rtimes D_{2}$%
}%
}%
\hfill\copy\matricesbox
}%
\ifdim\wd\onelinebox>\myboxwidth
\hbox to \myboxwidth{%
$\Pi_{7,7}$ spans $L_{9.7}$%
\hfil
$22|222\slashinfty2\rtimes D_{2}$%
}%
\box\matricesbox
\else
\hbox to \myboxwidth{%
\unhbox\onelinebox
}%
\fi
\else
\hbox to \myboxwidth{%
$\Pi_{7,7}$ spans $L_{9.7}$%
\hfil}%
\hbox to \myboxwidth{%
$22|222\slashinfty2\rtimes D_{2}$%
\hfil}%
\box\matricesbox
\fi
}%
\hfill\discretionary{}{}{}%
\setbox\matricesbox=\hbox{%
{$\left[\!\llap{\phantom{%
\begingroup \smaller\smaller\smaller\begin{tabular}{@{}c@{}}%
\phantom{0}\\\phantom{0}\\\phantom{0}
\end{tabular}\endgroup%
}}\right.$}%
\begingroup \smaller\smaller\smaller\begin{tabular}{@{}c@{}}%
-1\\\phantom{0}\\\phantom{0}
\end{tabular}\endgroup%
\kern3pt%
\begingroup \smaller\smaller\smaller\begin{tabular}{@{}c@{}}%
\phantom{0}\\3/2\\\phantom{0}
\end{tabular}\endgroup%
\kern3pt%
\begingroup \smaller\smaller\smaller\begin{tabular}{@{}c@{}}%
\phantom{0}\\\phantom{0}\\1/2
\end{tabular}\endgroup%
{$\left.\llap{\phantom{%
\begingroup \smaller\smaller\smaller\begin{tabular}{@{}c@{}}%
\phantom{0}\\\phantom{0}\\\phantom{0}
\end{tabular}\endgroup%
}}\!\right]$}%
{$\left[\!\llap{\phantom{%
\begingroup \smaller\smaller\smaller\begin{tabular}{@{}c@{}}%
0\\0\\0
\end{tabular}\endgroup%
}}\right.$}%
\begingroup \smaller\smaller\smaller\begin{tabular}{@{}c@{}}%
2\\-2\\0
\end{tabular}\endgroup%
\kern3pt%
\begingroup \smaller\smaller\smaller\begin{tabular}{@{}c@{}}%
2\\-1\\-3
\end{tabular}\endgroup%
\kern3pt%
\begingroup \smaller\smaller\smaller\begin{tabular}{@{}c@{}}%
6\\1\\-9
\end{tabular}\endgroup%
\kern3pt%
\begingroup \smaller\smaller\smaller\begin{tabular}{@{}c@{}}%
1\\1\\-1
\end{tabular}\endgroup%
{$\left.\llap{\phantom{%
\begingroup \smaller\smaller\smaller\begin{tabular}{@{}c@{}}%
0\\0\\0
\end{tabular}\endgroup%
}}\!\right]$}%
}%
\ifdim\wd\matricesbox>\halfwidth\myboxwidth=\hsize\else\myboxwidth=\halfwidth\fi
\vbox{%
\ifdim\myboxwidth=\hsize
\setbox\onelinebox=\hbox{%
\vbox{\hbox{%
$\Pi_{7,8}$ spans $L_{3.2}$%
}\hbox{%
$223|322\slashtwo\rtimes D_{2}$%
}%
}%
\hfill\copy\matricesbox
}%
\ifdim\wd\onelinebox>\myboxwidth
\hbox to \myboxwidth{%
$\Pi_{7,8}$ spans $L_{3.2}$%
\hfil
$223|322\slashtwo\rtimes D_{2}$%
}%
\box\matricesbox
\else
\hbox to \myboxwidth{%
\unhbox\onelinebox
}%
\fi
\else
\hbox to \myboxwidth{%
$\Pi_{7,8}$ spans $L_{3.2}$%
\hfil}%
\hbox to \myboxwidth{%
$223|322\slashtwo\rtimes D_{2}$%
\hfil}%
\box\matricesbox
\fi
}%
\hfill\discretionary{}{}{}%
\setbox\matricesbox=\hbox{%
{$\left[\!\llap{\phantom{%
\begingroup \smaller\smaller\smaller\begin{tabular}{@{}c@{}}%
\phantom{0}\\\phantom{0}\\\phantom{0}
\end{tabular}\endgroup%
}}\right.$}%
\begingroup \smaller\smaller\smaller\begin{tabular}{@{}c@{}}%
-1/3\\\phantom{0}\\\phantom{0}
\end{tabular}\endgroup%
\kern3pt%
\begingroup \smaller\smaller\smaller\begin{tabular}{@{}c@{}}%
\phantom{0}\\4/3\\\phantom{0}
\end{tabular}\endgroup%
\kern3pt%
\begingroup \smaller\smaller\smaller\begin{tabular}{@{}c@{}}%
\phantom{0}\\\phantom{0}\\6
\end{tabular}\endgroup%
{$\left.\llap{\phantom{%
\begingroup \smaller\smaller\smaller\begin{tabular}{@{}c@{}}%
\phantom{0}\\\phantom{0}\\\phantom{0}
\end{tabular}\endgroup%
}}\!\right]$}%
{$\left[\!\llap{\phantom{%
\begingroup \smaller\smaller\smaller\begin{tabular}{@{}c@{}}%
0\\0\\0
\end{tabular}\endgroup%
}}\right.$}%
\begingroup \smaller\smaller\smaller\begin{tabular}{@{}c@{}}%
6\\3\\-1
\end{tabular}\endgroup%
\kern3pt%
\begingroup \smaller\smaller\smaller\begin{tabular}{@{}c@{}}%
16\\2\\-4
\end{tabular}\endgroup%
\kern3pt%
\begingroup \smaller\smaller\smaller\begin{tabular}{@{}c@{}}%
8\\-2\\-2
\end{tabular}\endgroup%
\kern3pt%
\begingroup \smaller\smaller\smaller\begin{tabular}{@{}c@{}}%
1\\-1\\0
\end{tabular}\endgroup%
{$\left.\llap{\phantom{%
\begingroup \smaller\smaller\smaller\begin{tabular}{@{}c@{}}%
0\\0\\0
\end{tabular}\endgroup%
}}\!\right]$}%
}%
\ifdim\wd\matricesbox>\halfwidth\myboxwidth=\hsize\else\myboxwidth=\halfwidth\fi
\vbox{%
\ifdim\myboxwidth=\hsize
\setbox\onelinebox=\hbox{%
\vbox{\hbox{%
$\Pi_{7,9}$ spans $L_{150.13}$%
}\hbox{%
$22\slashinfty222|2\rtimes D_{2}$%
}%
}%
\hfill\copy\matricesbox
}%
\ifdim\wd\onelinebox>\myboxwidth
\hbox to \myboxwidth{%
$\Pi_{7,9}$ spans $L_{150.13}$%
\hfil
$22\slashinfty222|2\rtimes D_{2}$%
}%
\box\matricesbox
\else
\hbox to \myboxwidth{%
\unhbox\onelinebox
}%
\fi
\else
\hbox to \myboxwidth{%
$\Pi_{7,9}$ spans $L_{150.13}$%
\hfil}%
\hbox to \myboxwidth{%
$22\slashinfty222|2\rtimes D_{2}$%
\hfil}%
\box\matricesbox
\fi
}%
\hfill\discretionary{}{}{}%
\setbox\matricesbox=\hbox{%
{$\left[\!\llap{\phantom{%
\begingroup \smaller\smaller\smaller\begin{tabular}{@{}c@{}}%
\phantom{0}\\\phantom{0}\\\phantom{0}
\end{tabular}\endgroup%
}}\right.$}%
\begingroup \smaller\smaller\smaller\begin{tabular}{@{}c@{}}%
-3/7\\\phantom{0}\\\phantom{0}
\end{tabular}\endgroup%
\kern3pt%
\begingroup \smaller\smaller\smaller\begin{tabular}{@{}c@{}}%
\phantom{0}\\3/7\\\phantom{0}
\end{tabular}\endgroup%
\kern3pt%
\begingroup \smaller\smaller\smaller\begin{tabular}{@{}c@{}}%
\phantom{0}\\\phantom{0}\\7
\end{tabular}\endgroup%
{$\left.\llap{\phantom{%
\begingroup \smaller\smaller\smaller\begin{tabular}{@{}c@{}}%
\phantom{0}\\\phantom{0}\\\phantom{0}
\end{tabular}\endgroup%
}}\!\right]$}%
{$\left[\!\llap{\phantom{%
\begingroup \smaller\smaller\smaller\begin{tabular}{@{}c@{}}%
0\\0\\0
\end{tabular}\endgroup%
}}\right.$}%
\begingroup \smaller\smaller\smaller\begin{tabular}{@{}c@{}}%
3\\-4\\0
\end{tabular}\endgroup%
\kern3pt%
\begingroup \smaller\smaller\smaller\begin{tabular}{@{}c@{}}%
4\\-3\\1
\end{tabular}\endgroup%
\kern3pt%
\begingroup \smaller\smaller\smaller\begin{tabular}{@{}c@{}}%
12\\5\\3
\end{tabular}\endgroup%
\kern3pt%
\begingroup \smaller\smaller\smaller\begin{tabular}{@{}c@{}}%
7\\7\\1
\end{tabular}\endgroup%
{$\left.\llap{\phantom{%
\begingroup \smaller\smaller\smaller\begin{tabular}{@{}c@{}}%
0\\0\\0
\end{tabular}\endgroup%
}}\!\right]$}%
}%
\ifdim\wd\matricesbox>\halfwidth\myboxwidth=\hsize\else\myboxwidth=\halfwidth\fi
\vbox{%
\ifdim\myboxwidth=\hsize
\setbox\onelinebox=\hbox{%
\vbox{\hbox{%
$\Pi_{7,10}$ spans $L_{44.7}$%
}\hbox{%
$2|262\slashinfty26\rtimes D_{2}$%
}%
}%
\hfill\copy\matricesbox
}%
\ifdim\wd\onelinebox>\myboxwidth
\hbox to \myboxwidth{%
$\Pi_{7,10}$ spans $L_{44.7}$%
\hfil
$2|262\slashinfty26\rtimes D_{2}$%
}%
\box\matricesbox
\else
\hbox to \myboxwidth{%
\unhbox\onelinebox
}%
\fi
\else
\hbox to \myboxwidth{%
$\Pi_{7,10}$ spans $L_{44.7}$%
\hfil}%
\hbox to \myboxwidth{%
$2|262\slashinfty26\rtimes D_{2}$%
\hfil}%
\box\matricesbox
\fi
}%
\hfill\discretionary{}{}{}%
\setbox\matricesbox=\hbox{%
{$\left[\!\llap{\phantom{%
\begingroup \smaller\smaller\smaller\begin{tabular}{@{}c@{}}%
\phantom{0}\\\phantom{0}\\\phantom{0}
\end{tabular}\endgroup%
}}\right.$}%
\begingroup \smaller\smaller\smaller\begin{tabular}{@{}c@{}}%
-1\\\phantom{0}\\\phantom{0}
\end{tabular}\endgroup%
\kern3pt%
\begingroup \smaller\smaller\smaller\begin{tabular}{@{}c@{}}%
\phantom{0}\\3/2\\\phantom{0}
\end{tabular}\endgroup%
\kern3pt%
\begingroup \smaller\smaller\smaller\begin{tabular}{@{}c@{}}%
\phantom{0}\\\phantom{0}\\9/2
\end{tabular}\endgroup%
{$\left.\llap{\phantom{%
\begingroup \smaller\smaller\smaller\begin{tabular}{@{}c@{}}%
\phantom{0}\\\phantom{0}\\\phantom{0}
\end{tabular}\endgroup%
}}\!\right]$}%
{$\left[\!\llap{\phantom{%
\begingroup \smaller\smaller\smaller\begin{tabular}{@{}c@{}}%
0\\0\\0
\end{tabular}\endgroup%
}}\right.$}%
\begingroup \smaller\smaller\smaller\begin{tabular}{@{}c@{}}%
2\\-2\\0
\end{tabular}\endgroup%
\kern3pt%
\begingroup \smaller\smaller\smaller\begin{tabular}{@{}c@{}}%
2\\-1\\1
\end{tabular}\endgroup%
\kern3pt%
\begingroup \smaller\smaller\smaller\begin{tabular}{@{}c@{}}%
2\\1\\1
\end{tabular}\endgroup%
\kern3pt%
\begingroup \smaller\smaller\smaller\begin{tabular}{@{}c@{}}%
6\\5\\1
\end{tabular}\endgroup%
{$\left.\llap{\phantom{%
\begingroup \smaller\smaller\smaller\begin{tabular}{@{}c@{}}%
0\\0\\0
\end{tabular}\endgroup%
}}\!\right]$}%
}%
\ifdim\wd\matricesbox>\halfwidth\myboxwidth=\hsize\else\myboxwidth=\halfwidth\fi
\vbox{%
\ifdim\myboxwidth=\hsize
\setbox\onelinebox=\hbox{%
\vbox{\hbox{%
$\Pi_{7,11}$ spans $L_{155.1}$%
}\hbox{%
$33|332\slashthree2\rtimes D_{2}$%
}%
}%
\hfill\copy\matricesbox
}%
\ifdim\wd\onelinebox>\myboxwidth
\hbox to \myboxwidth{%
$\Pi_{7,11}$ spans $L_{155.1}$%
\hfil
$33|332\slashthree2\rtimes D_{2}$%
}%
\box\matricesbox
\else
\hbox to \myboxwidth{%
\unhbox\onelinebox
}%
\fi
\else
\hbox to \myboxwidth{%
$\Pi_{7,11}$ spans $L_{155.1}$%
\hfil}%
\hbox to \myboxwidth{%
$33|332\slashthree2\rtimes D_{2}$%
\hfil}%
\box\matricesbox
\fi
}%
\hfill\discretionary{}{}{}%
\setbox\matricesbox=\hbox{%
{$\left[\!\llap{\phantom{%
\begingroup \smaller\smaller\smaller\begin{tabular}{@{}c@{}}%
\phantom{0}\\\phantom{0}\\\phantom{0}
\end{tabular}\endgroup%
}}\right.$}%
\begingroup \smaller\smaller\smaller\begin{tabular}{@{}c@{}}%
-3\\\phantom{0}\\\phantom{0}
\end{tabular}\endgroup%
\kern3pt%
\begingroup \smaller\smaller\smaller\begin{tabular}{@{}c@{}}%
\phantom{0}\\1\\\phantom{0}
\end{tabular}\endgroup%
\kern3pt%
\begingroup \smaller\smaller\smaller\begin{tabular}{@{}c@{}}%
\phantom{0}\\\phantom{0}\\3
\end{tabular}\endgroup%
{$\left.\llap{\phantom{%
\begingroup \smaller\smaller\smaller\begin{tabular}{@{}c@{}}%
\phantom{0}\\\phantom{0}\\\phantom{0}
\end{tabular}\endgroup%
}}\!\right]$}%
{$\left[\!\llap{\phantom{%
\begingroup \smaller\smaller\smaller\begin{tabular}{@{}c@{}}%
0\\0\\0
\end{tabular}\endgroup%
}}\right.$}%
\begingroup \smaller\smaller\smaller\begin{tabular}{@{}c@{}}%
4\\7\\1
\end{tabular}\endgroup%
\kern3pt%
\begingroup \smaller\smaller\smaller\begin{tabular}{@{}c@{}}%
1\\1\\1
\end{tabular}\endgroup%
\kern3pt%
\begingroup \smaller\smaller\smaller\begin{tabular}{@{}c@{}}%
1\\-1\\1
\end{tabular}\endgroup%
\kern3pt%
\begingroup \smaller\smaller\smaller\begin{tabular}{@{}c@{}}%
1\\-2\\0
\end{tabular}\endgroup%
{$\left.\llap{\phantom{%
\begingroup \smaller\smaller\smaller\begin{tabular}{@{}c@{}}%
0\\0\\0
\end{tabular}\endgroup%
}}\!\right]$}%
}%
\ifdim\wd\matricesbox>\halfwidth\myboxwidth=\hsize\else\myboxwidth=\halfwidth\fi
\vbox{%
\ifdim\myboxwidth=\hsize
\setbox\onelinebox=\hbox{%
\vbox{\hbox{%
$\Pi_{7,12}=\hbox{GN}_{47}$ spans $L_{7.7}$%
}\hbox{%
$\slashthree\infty\infty\infty|\infty\infty\infty\rtimes D_{2}$%
}%
}%
\hfill\copy\matricesbox
}%
\ifdim\wd\onelinebox>\myboxwidth
\hbox to \myboxwidth{%
$\Pi_{7,12}=\hbox{GN}_{47}$ spans $L_{7.7}$%
\hfil
$\slashthree\infty\infty\infty|\infty\infty\infty\rtimes D_{2}$%
}%
\box\matricesbox
\else
\hbox to \myboxwidth{%
\unhbox\onelinebox
}%
\fi
\else
\hbox to \myboxwidth{%
$\Pi_{7,12}=\hbox{GN}_{47}$ spans $L_{7.7}$%
\hfil}%
\hbox to \myboxwidth{%
$\slashthree\infty\infty\infty|\infty\infty\infty\rtimes D_{2}$%
\hfil}%
\box\matricesbox
\fi
}%
\hfill\discretionary{}{}{}%
\setbox\matricesbox=\hbox{%
{$\left[\!\llap{\phantom{%
\begingroup \smaller\smaller\smaller\begin{tabular}{@{}c@{}}%
\phantom{0}\\\phantom{0}\\\phantom{0}
\end{tabular}\endgroup%
}}\right.$}%
\begingroup \smaller\smaller\smaller\begin{tabular}{@{}c@{}}%
-1\\\phantom{0}\\\phantom{0}
\end{tabular}\endgroup%
\kern3pt%
\begingroup \smaller\smaller\smaller\begin{tabular}{@{}c@{}}%
\phantom{0}\\2\\\phantom{0}
\end{tabular}\endgroup%
\kern3pt%
\begingroup \smaller\smaller\smaller\begin{tabular}{@{}c@{}}%
\phantom{0}\\\phantom{0}\\2
\end{tabular}\endgroup%
{$\left.\llap{\phantom{%
\begingroup \smaller\smaller\smaller\begin{tabular}{@{}c@{}}%
\phantom{0}\\\phantom{0}\\\phantom{0}
\end{tabular}\endgroup%
}}\!\right]$}%
{$\left[\!\llap{\phantom{%
\begingroup \smaller\smaller\smaller\begin{tabular}{@{}c@{}}%
0\\0\\0
\end{tabular}\endgroup%
}}\right.$}%
\begingroup \smaller\smaller\smaller\begin{tabular}{@{}c@{}}%
1\\-1\\0
\end{tabular}\endgroup%
\kern3pt%
\begingroup \smaller\smaller\smaller\begin{tabular}{@{}c@{}}%
4\\-1\\3
\end{tabular}\endgroup%
\kern3pt%
\begingroup \smaller\smaller\smaller\begin{tabular}{@{}c@{}}%
4\\1\\3
\end{tabular}\endgroup%
\kern3pt%
\begingroup \smaller\smaller\smaller\begin{tabular}{@{}c@{}}%
4\\3\\1
\end{tabular}\endgroup%
{$\left.\llap{\phantom{%
\begingroup \smaller\smaller\smaller\begin{tabular}{@{}c@{}}%
0\\0\\0
\end{tabular}\endgroup%
}}\!\right]$}%
}%
\ifdim\wd\matricesbox>\halfwidth\myboxwidth=\hsize\else\myboxwidth=\halfwidth\fi
\vbox{%
\ifdim\myboxwidth=\hsize
\setbox\onelinebox=\hbox{%
\vbox{\hbox{%
$\Pi_{7,13}=\hbox{GN}_{40}$ spans $L_{1.9}$%
}\hbox{%
$|\infty2\infty\slashtwo\infty2\infty\rtimes D_{2}$%
}%
}%
\hfill\copy\matricesbox
}%
\ifdim\wd\onelinebox>\myboxwidth
\hbox to \myboxwidth{%
$\Pi_{7,13}=\hbox{GN}_{40}$ spans $L_{1.9}$%
\hfil
$|\infty2\infty\slashtwo\infty2\infty\rtimes D_{2}$%
}%
\box\matricesbox
\else
\hbox to \myboxwidth{%
\unhbox\onelinebox
}%
\fi
\else
\hbox to \myboxwidth{%
$\Pi_{7,13}=\hbox{GN}_{40}$ spans $L_{1.9}$%
\hfil}%
\hbox to \myboxwidth{%
$|\infty2\infty\slashtwo\infty2\infty\rtimes D_{2}$%
\hfil}%
\box\matricesbox
\fi
}%
\hfill\discretionary{}{}{}%
\setbox\matricesbox=\hbox{%
{$\left[\!\llap{\phantom{%
\begingroup \smaller\smaller\smaller\begin{tabular}{@{}c@{}}%
\phantom{0}\\\phantom{0}\\\phantom{0}
\end{tabular}\endgroup%
}}\right.$}%
\begingroup \smaller\smaller\smaller\begin{tabular}{@{}c@{}}%
-1/8\\\phantom{0}\\\phantom{0}
\end{tabular}\endgroup%
\kern3pt%
\begingroup \smaller\smaller\smaller\begin{tabular}{@{}c@{}}%
\phantom{0}\\21/2\\\phantom{0}
\end{tabular}\endgroup%
\kern3pt%
\begingroup \smaller\smaller\smaller\begin{tabular}{@{}c@{}}%
\phantom{0}\\\phantom{0}\\28
\end{tabular}\endgroup%
{$\left.\llap{\phantom{%
\begingroup \smaller\smaller\smaller\begin{tabular}{@{}c@{}}%
\phantom{0}\\\phantom{0}\\\phantom{0}
\end{tabular}\endgroup%
}}\!\right]$}%
{$\left[\!\llap{\phantom{%
\begingroup \smaller\smaller\smaller\begin{tabular}{@{}c@{}}%
0\\0\\0
\end{tabular}\endgroup%
}}\right.$}%
\begingroup \smaller\smaller\smaller\begin{tabular}{@{}c@{}}%
6\\-1\\0
\end{tabular}\endgroup%
\kern3pt%
\begingroup \smaller\smaller\smaller\begin{tabular}{@{}c@{}}%
112\\-8\\-6
\end{tabular}\endgroup%
\kern3pt%
\begingroup \smaller\smaller\smaller\begin{tabular}{@{}c@{}}%
42\\-1\\-3
\end{tabular}\endgroup%
\kern3pt%
\begingroup \smaller\smaller\smaller\begin{tabular}{@{}c@{}}%
14\\1\\-1
\end{tabular}\endgroup%
\kern3pt%
\begingroup \smaller\smaller\smaller\begin{tabular}{@{}c@{}}%
6\\1\\0
\end{tabular}\endgroup%
\kern3pt%
\begingroup \smaller\smaller\smaller\begin{tabular}{@{}c@{}}%
14\\1\\1
\end{tabular}\endgroup%
\kern3pt%
\begingroup \smaller\smaller\smaller\begin{tabular}{@{}c@{}}%
14\\-1\\1
\end{tabular}\endgroup%
{$\left.\llap{\phantom{%
\begingroup \smaller\smaller\smaller\begin{tabular}{@{}c@{}}%
0\\0\\0
\end{tabular}\endgroup%
}}\!\right]$}%
}%
\ifdim\wd\matricesbox>\halfwidth\myboxwidth=\hsize\else\myboxwidth=\halfwidth\fi
\vbox{%
\ifdim\myboxwidth=\hsize
\setbox\onelinebox=\hbox{%
\vbox{\hbox{%
$\Pi_{7,14}$ spans $L_{22.8}$%
}\hbox{%
$2222232$%
}%
}%
\hfill\copy\matricesbox
}%
\ifdim\wd\onelinebox>\myboxwidth
\hbox to \myboxwidth{%
$\Pi_{7,14}$ spans $L_{22.8}$%
\hfil
$2222232$%
}%
\box\matricesbox
\else
\hbox to \myboxwidth{%
\unhbox\onelinebox
}%
\fi
\else
\hbox to \myboxwidth{%
$\Pi_{7,14}$ spans $L_{22.8}$%
\hfil}%
\hbox to \myboxwidth{%
$2222232$%
\hfil}%
\box\matricesbox
\fi
}%
\hfill\discretionary{}{}{}%
\setbox\matricesbox=\hbox{%
{$\left[\!\llap{\phantom{%
\begingroup \smaller\smaller\smaller\begin{tabular}{@{}c@{}}%
\phantom{0}\\\phantom{0}\\\phantom{0}
\end{tabular}\endgroup%
}}\right.$}%
\begingroup \smaller\smaller\smaller\begin{tabular}{@{}c@{}}%
-1/2\\\phantom{0}\\\phantom{0}
\end{tabular}\endgroup%
\kern3pt%
\begingroup \smaller\smaller\smaller\begin{tabular}{@{}c@{}}%
\phantom{0}\\4\\-1
\end{tabular}\endgroup%
\kern3pt%
\begingroup \smaller\smaller\smaller\begin{tabular}{@{}c@{}}%
\phantom{0}\\-1\\4
\end{tabular}\endgroup%
{$\left.\llap{\phantom{%
\begingroup \smaller\smaller\smaller\begin{tabular}{@{}c@{}}%
\phantom{0}\\\phantom{0}\\\phantom{0}
\end{tabular}\endgroup%
}}\!\right]$}%
{$\left[\!\llap{\phantom{%
\begingroup \smaller\smaller\smaller\begin{tabular}{@{}c@{}}%
0\\0\\0
\end{tabular}\endgroup%
}}\right.$}%
\begingroup \smaller\smaller\smaller\begin{tabular}{@{}c@{}}%
2\\0\\-1
\end{tabular}\endgroup%
\kern3pt%
\begingroup \smaller\smaller\smaller\begin{tabular}{@{}c@{}}%
6\\2\\-1
\end{tabular}\endgroup%
\kern3pt%
\begingroup \smaller\smaller\smaller\begin{tabular}{@{}c@{}}%
10\\4\\1
\end{tabular}\endgroup%
\kern3pt%
\begingroup \smaller\smaller\smaller\begin{tabular}{@{}c@{}}%
6\\2\\2
\end{tabular}\endgroup%
\kern3pt%
\begingroup \smaller\smaller\smaller\begin{tabular}{@{}c@{}}%
2\\0\\1
\end{tabular}\endgroup%
\kern3pt%
\begingroup \smaller\smaller\smaller\begin{tabular}{@{}c@{}}%
2\\-1\\0
\end{tabular}\endgroup%
\kern3pt%
\begingroup \smaller\smaller\smaller\begin{tabular}{@{}c@{}}%
6\\-2\\-2
\end{tabular}\endgroup%
{$\left.\llap{\phantom{%
\begingroup \smaller\smaller\smaller\begin{tabular}{@{}c@{}}%
0\\0\\0
\end{tabular}\endgroup%
}}\!\right]$}%
}%
\ifdim\wd\matricesbox>\halfwidth\myboxwidth=\hsize\else\myboxwidth=\halfwidth\fi
\vbox{%
\ifdim\myboxwidth=\hsize
\setbox\onelinebox=\hbox{%
\vbox{\hbox{%
$\Pi_{7,15}$ spans $L_{31.2}$%
}\hbox{%
$2222322$%
}%
}%
\hfill\copy\matricesbox
}%
\ifdim\wd\onelinebox>\myboxwidth
\hbox to \myboxwidth{%
$\Pi_{7,15}$ spans $L_{31.2}$%
\hfil
$2222322$%
}%
\box\matricesbox
\else
\hbox to \myboxwidth{%
\unhbox\onelinebox
}%
\fi
\else
\hbox to \myboxwidth{%
$\Pi_{7,15}$ spans $L_{31.2}$%
\hfil}%
\hbox to \myboxwidth{%
$2222322$%
\hfil}%
\box\matricesbox
\fi
}%
\hfill\discretionary{}{}{}%
\setbox\matricesbox=\hbox{%
{$\left[\!\llap{\phantom{%
\begingroup \smaller\smaller\smaller\begin{tabular}{@{}c@{}}%
\phantom{0}\\\phantom{0}\\\phantom{0}
\end{tabular}\endgroup%
}}\right.$}%
\begingroup \smaller\smaller\smaller\begin{tabular}{@{}c@{}}%
-1/8\\\phantom{0}\\\phantom{0}
\end{tabular}\endgroup%
\kern3pt%
\begingroup \smaller\smaller\smaller\begin{tabular}{@{}c@{}}%
\phantom{0}\\45/2\\-15/2
\end{tabular}\endgroup%
\kern3pt%
\begingroup \smaller\smaller\smaller\begin{tabular}{@{}c@{}}%
\phantom{0}\\-15/2\\45/2
\end{tabular}\endgroup%
{$\left.\llap{\phantom{%
\begingroup \smaller\smaller\smaller\begin{tabular}{@{}c@{}}%
\phantom{0}\\\phantom{0}\\\phantom{0}
\end{tabular}\endgroup%
}}\!\right]$}%
{$\left[\!\llap{\phantom{%
\begingroup \smaller\smaller\smaller\begin{tabular}{@{}c@{}}%
0\\0\\0
\end{tabular}\endgroup%
}}\right.$}%
\begingroup \smaller\smaller\smaller\begin{tabular}{@{}c@{}}%
10\\0\\-1
\end{tabular}\endgroup%
\kern3pt%
\begingroup \smaller\smaller\smaller\begin{tabular}{@{}c@{}}%
12\\-1\\-1
\end{tabular}\endgroup%
\kern3pt%
\begingroup \smaller\smaller\smaller\begin{tabular}{@{}c@{}}%
10\\-1\\0
\end{tabular}\endgroup%
\kern3pt%
\begingroup \smaller\smaller\smaller\begin{tabular}{@{}c@{}}%
30\\-1\\2
\end{tabular}\endgroup%
\kern3pt%
\begingroup \smaller\smaller\smaller\begin{tabular}{@{}c@{}}%
60\\1\\5
\end{tabular}\endgroup%
\kern3pt%
\begingroup \smaller\smaller\smaller\begin{tabular}{@{}c@{}}%
12\\1\\1
\end{tabular}\endgroup%
\kern3pt%
\begingroup \smaller\smaller\smaller\begin{tabular}{@{}c@{}}%
10\\1\\0
\end{tabular}\endgroup%
{$\left.\llap{\phantom{%
\begingroup \smaller\smaller\smaller\begin{tabular}{@{}c@{}}%
0\\0\\0
\end{tabular}\endgroup%
}}\!\right]$}%
}%
\ifdim\wd\matricesbox>\halfwidth\myboxwidth=\hsize\else\myboxwidth=\halfwidth\fi
\vbox{%
\ifdim\myboxwidth=\hsize
\setbox\onelinebox=\hbox{%
\vbox{\hbox{%
$\Pi_{7,16}$ spans $L_{16.16}$%
}\hbox{%
$2222223$%
}%
}%
\hfill\copy\matricesbox
}%
\ifdim\wd\onelinebox>\myboxwidth
\hbox to \myboxwidth{%
$\Pi_{7,16}$ spans $L_{16.16}$%
\hfil
$2222223$%
}%
\box\matricesbox
\else
\hbox to \myboxwidth{%
\unhbox\onelinebox
}%
\fi
\else
\hbox to \myboxwidth{%
$\Pi_{7,16}$ spans $L_{16.16}$%
\hfil}%
\hbox to \myboxwidth{%
$2222223$%
\hfil}%
\box\matricesbox
\fi
}%
\hfill\discretionary{}{}{}%
\setbox\matricesbox=\hbox{%
{$\left[\!\llap{\phantom{%
\begingroup \smaller\smaller\smaller\begin{tabular}{@{}c@{}}%
\phantom{0}\\\phantom{0}\\\phantom{0}
\end{tabular}\endgroup%
}}\right.$}%
\begingroup \smaller\smaller\smaller\begin{tabular}{@{}c@{}}%
-1\\\phantom{0}\\\phantom{0}
\end{tabular}\endgroup%
\kern3pt%
\begingroup \smaller\smaller\smaller\begin{tabular}{@{}c@{}}%
\phantom{0}\\2\\\phantom{0}
\end{tabular}\endgroup%
\kern3pt%
\begingroup \smaller\smaller\smaller\begin{tabular}{@{}c@{}}%
\phantom{0}\\\phantom{0}\\4
\end{tabular}\endgroup%
{$\left.\llap{\phantom{%
\begingroup \smaller\smaller\smaller\begin{tabular}{@{}c@{}}%
\phantom{0}\\\phantom{0}\\\phantom{0}
\end{tabular}\endgroup%
}}\!\right]$}%
{$\left[\!\llap{\phantom{%
\begingroup \smaller\smaller\smaller\begin{tabular}{@{}c@{}}%
0\\0\\0
\end{tabular}\endgroup%
}}\right.$}%
\begingroup \smaller\smaller\smaller\begin{tabular}{@{}c@{}}%
1\\1\\0
\end{tabular}\endgroup%
\kern3pt%
\begingroup \smaller\smaller\smaller\begin{tabular}{@{}c@{}}%
16\\8\\-6
\end{tabular}\endgroup%
\kern3pt%
\begingroup \smaller\smaller\smaller\begin{tabular}{@{}c@{}}%
8\\2\\-4
\end{tabular}\endgroup%
\kern3pt%
\begingroup \smaller\smaller\smaller\begin{tabular}{@{}c@{}}%
2\\-1\\-1
\end{tabular}\endgroup%
\kern3pt%
\begingroup \smaller\smaller\smaller\begin{tabular}{@{}c@{}}%
1\\-1\\0
\end{tabular}\endgroup%
\kern3pt%
\begingroup \smaller\smaller\smaller\begin{tabular}{@{}c@{}}%
2\\-1\\1
\end{tabular}\endgroup%
\kern3pt%
\begingroup \smaller\smaller\smaller\begin{tabular}{@{}c@{}}%
2\\1\\1
\end{tabular}\endgroup%
{$\left.\llap{\phantom{%
\begingroup \smaller\smaller\smaller\begin{tabular}{@{}c@{}}%
0\\0\\0
\end{tabular}\endgroup%
}}\!\right]$}%
}%
\ifdim\wd\matricesbox>\halfwidth\myboxwidth=\hsize\else\myboxwidth=\halfwidth\fi
\vbox{%
\ifdim\myboxwidth=\hsize
\setbox\onelinebox=\hbox{%
\vbox{\hbox{%
$\Pi_{7,17}$ spans $L_{141.3}$%
}\hbox{%
$22\infty22\infty2$%
}%
}%
\hfill\copy\matricesbox
}%
\ifdim\wd\onelinebox>\myboxwidth
\hbox to \myboxwidth{%
$\Pi_{7,17}$ spans $L_{141.3}$%
\hfil
$22\infty22\infty2$%
}%
\box\matricesbox
\else
\hbox to \myboxwidth{%
\unhbox\onelinebox
}%
\fi
\else
\hbox to \myboxwidth{%
$\Pi_{7,17}$ spans $L_{141.3}$%
\hfil}%
\hbox to \myboxwidth{%
$22\infty22\infty2$%
\hfil}%
\box\matricesbox
\fi
}%
\hfill\discretionary{}{}{}%
\setbox\matricesbox=\hbox{%
{$\left[\!\llap{\phantom{%
\begingroup \smaller\smaller\smaller\begin{tabular}{@{}c@{}}%
\phantom{0}\\\phantom{0}\\\phantom{0}
\end{tabular}\endgroup%
}}\right.$}%
\begingroup \smaller\smaller\smaller\begin{tabular}{@{}c@{}}%
-1\\\phantom{0}\\\phantom{0}
\end{tabular}\endgroup%
\kern3pt%
\begingroup \smaller\smaller\smaller\begin{tabular}{@{}c@{}}%
\phantom{0}\\2\\\phantom{0}
\end{tabular}\endgroup%
\kern3pt%
\begingroup \smaller\smaller\smaller\begin{tabular}{@{}c@{}}%
\phantom{0}\\\phantom{0}\\4
\end{tabular}\endgroup%
{$\left.\llap{\phantom{%
\begingroup \smaller\smaller\smaller\begin{tabular}{@{}c@{}}%
\phantom{0}\\\phantom{0}\\\phantom{0}
\end{tabular}\endgroup%
}}\!\right]$}%
{$\left[\!\llap{\phantom{%
\begingroup \smaller\smaller\smaller\begin{tabular}{@{}c@{}}%
0\\0\\0
\end{tabular}\endgroup%
}}\right.$}%
\begingroup \smaller\smaller\smaller\begin{tabular}{@{}c@{}}%
1\\1\\0
\end{tabular}\endgroup%
\kern3pt%
\begingroup \smaller\smaller\smaller\begin{tabular}{@{}c@{}}%
16\\8\\-6
\end{tabular}\endgroup%
\kern3pt%
\begingroup \smaller\smaller\smaller\begin{tabular}{@{}c@{}}%
8\\2\\-4
\end{tabular}\endgroup%
\kern3pt%
\begingroup \smaller\smaller\smaller\begin{tabular}{@{}c@{}}%
2\\-1\\-1
\end{tabular}\endgroup%
\kern3pt%
\begingroup \smaller\smaller\smaller\begin{tabular}{@{}c@{}}%
8\\-6\\0
\end{tabular}\endgroup%
\kern3pt%
\begingroup \smaller\smaller\smaller\begin{tabular}{@{}c@{}}%
2\\-1\\1
\end{tabular}\endgroup%
\kern3pt%
\begingroup \smaller\smaller\smaller\begin{tabular}{@{}c@{}}%
2\\1\\1
\end{tabular}\endgroup%
{$\left.\llap{\phantom{%
\begingroup \smaller\smaller\smaller\begin{tabular}{@{}c@{}}%
0\\0\\0
\end{tabular}\endgroup%
}}\!\right]$}%
}%
\ifdim\wd\matricesbox>\halfwidth\myboxwidth=\hsize\else\myboxwidth=\halfwidth\fi
\vbox{%
\ifdim\myboxwidth=\hsize
\setbox\onelinebox=\hbox{%
\vbox{\hbox{%
$\Pi_{7,18}$ spans $L_{141.3}$%
}\hbox{%
$22\infty\infty\infty\infty2$%
}%
}%
\hfill\copy\matricesbox
}%
\ifdim\wd\onelinebox>\myboxwidth
\hbox to \myboxwidth{%
$\Pi_{7,18}$ spans $L_{141.3}$%
\hfil
$22\infty\infty\infty\infty2$%
}%
\box\matricesbox
\else
\hbox to \myboxwidth{%
\unhbox\onelinebox
}%
\fi
\else
\hbox to \myboxwidth{%
$\Pi_{7,18}$ spans $L_{141.3}$%
\hfil}%
\hbox to \myboxwidth{%
$22\infty\infty\infty\infty2$%
\hfil}%
\box\matricesbox
\fi
}%
\hfill\discretionary{}{}{}%
\setbox\matricesbox=\hbox{%
{$\left[\!\llap{\phantom{%
\begingroup \smaller\smaller\smaller\begin{tabular}{@{}c@{}}%
\phantom{0}\\\phantom{0}\\\phantom{0}
\end{tabular}\endgroup%
}}\right.$}%
\begingroup \smaller\smaller\smaller\begin{tabular}{@{}c@{}}%
-1/2\\\phantom{0}\\\phantom{0}
\end{tabular}\endgroup%
\kern3pt%
\begingroup \smaller\smaller\smaller\begin{tabular}{@{}c@{}}%
\phantom{0}\\15/2\\-3/2
\end{tabular}\endgroup%
\kern3pt%
\begingroup \smaller\smaller\smaller\begin{tabular}{@{}c@{}}%
\phantom{0}\\-3/2\\15/2
\end{tabular}\endgroup%
{$\left.\llap{\phantom{%
\begingroup \smaller\smaller\smaller\begin{tabular}{@{}c@{}}%
\phantom{0}\\\phantom{0}\\\phantom{0}
\end{tabular}\endgroup%
}}\!\right]$}%
{$\left[\!\llap{\phantom{%
\begingroup \smaller\smaller\smaller\begin{tabular}{@{}c@{}}%
0\\0\\0
\end{tabular}\endgroup%
}}\right.$}%
\begingroup \smaller\smaller\smaller\begin{tabular}{@{}c@{}}%
4\\-1\\-1
\end{tabular}\endgroup%
\kern3pt%
\begingroup \smaller\smaller\smaller\begin{tabular}{@{}c@{}}%
18\\-1\\-5
\end{tabular}\endgroup%
\kern3pt%
\begingroup \smaller\smaller\smaller\begin{tabular}{@{}c@{}}%
12\\1\\-3
\end{tabular}\endgroup%
\kern3pt%
\begingroup \smaller\smaller\smaller\begin{tabular}{@{}c@{}}%
3\\1\\0
\end{tabular}\endgroup%
\kern3pt%
\begingroup \smaller\smaller\smaller\begin{tabular}{@{}c@{}}%
4\\1\\1
\end{tabular}\endgroup%
\kern3pt%
\begingroup \smaller\smaller\smaller\begin{tabular}{@{}c@{}}%
3\\0\\1
\end{tabular}\endgroup%
\kern3pt%
\begingroup \smaller\smaller\smaller\begin{tabular}{@{}c@{}}%
3\\-1\\0
\end{tabular}\endgroup%
{$\left.\llap{\phantom{%
\begingroup \smaller\smaller\smaller\begin{tabular}{@{}c@{}}%
0\\0\\0
\end{tabular}\endgroup%
}}\!\right]$}%
}%
\ifdim\wd\matricesbox>\halfwidth\myboxwidth=\hsize\else\myboxwidth=\halfwidth\fi
\vbox{%
\ifdim\myboxwidth=\hsize
\setbox\onelinebox=\hbox{%
\vbox{\hbox{%
$\Pi_{7,19}$ spans $L_{4.12}$%
}\hbox{%
$22\infty22\infty2$%
}%
}%
\hfill\copy\matricesbox
}%
\ifdim\wd\onelinebox>\myboxwidth
\hbox to \myboxwidth{%
$\Pi_{7,19}$ spans $L_{4.12}$%
\hfil
$22\infty22\infty2$%
}%
\box\matricesbox
\else
\hbox to \myboxwidth{%
\unhbox\onelinebox
}%
\fi
\else
\hbox to \myboxwidth{%
$\Pi_{7,19}$ spans $L_{4.12}$%
\hfil}%
\hbox to \myboxwidth{%
$22\infty22\infty2$%
\hfil}%
\box\matricesbox
\fi
}%
\hfill\discretionary{}{}{}%
\setbox\matricesbox=\hbox{%
{$\left[\!\llap{\phantom{%
\begingroup \smaller\smaller\smaller\begin{tabular}{@{}c@{}}%
\phantom{0}\\\phantom{0}\\\phantom{0}
\end{tabular}\endgroup%
}}\right.$}%
\begingroup \smaller\smaller\smaller\begin{tabular}{@{}c@{}}%
-1/8\\\phantom{0}\\\phantom{0}
\end{tabular}\endgroup%
\kern3pt%
\begingroup \smaller\smaller\smaller\begin{tabular}{@{}c@{}}%
\phantom{0}\\45/2\\-5/2
\end{tabular}\endgroup%
\kern3pt%
\begingroup \smaller\smaller\smaller\begin{tabular}{@{}c@{}}%
\phantom{0}\\-5/2\\45/2
\end{tabular}\endgroup%
{$\left.\llap{\phantom{%
\begingroup \smaller\smaller\smaller\begin{tabular}{@{}c@{}}%
\phantom{0}\\\phantom{0}\\\phantom{0}
\end{tabular}\endgroup%
}}\!\right]$}%
{$\left[\!\llap{\phantom{%
\begingroup \smaller\smaller\smaller\begin{tabular}{@{}c@{}}%
0\\0\\0
\end{tabular}\endgroup%
}}\right.$}%
\begingroup \smaller\smaller\smaller\begin{tabular}{@{}c@{}}%
10\\0\\-1
\end{tabular}\endgroup%
\kern3pt%
\begingroup \smaller\smaller\smaller\begin{tabular}{@{}c@{}}%
32\\-2\\-2
\end{tabular}\endgroup%
\kern3pt%
\begingroup \smaller\smaller\smaller\begin{tabular}{@{}c@{}}%
10\\-1\\0
\end{tabular}\endgroup%
\kern3pt%
\begingroup \smaller\smaller\smaller\begin{tabular}{@{}c@{}}%
10\\0\\1
\end{tabular}\endgroup%
\kern3pt%
\begingroup \smaller\smaller\smaller\begin{tabular}{@{}c@{}}%
32\\2\\2
\end{tabular}\endgroup%
\kern3pt%
\begingroup \smaller\smaller\smaller\begin{tabular}{@{}c@{}}%
50\\4\\1
\end{tabular}\endgroup%
\kern3pt%
\begingroup \smaller\smaller\smaller\begin{tabular}{@{}c@{}}%
40\\3\\-1
\end{tabular}\endgroup%
{$\left.\llap{\phantom{%
\begingroup \smaller\smaller\smaller\begin{tabular}{@{}c@{}}%
0\\0\\0
\end{tabular}\endgroup%
}}\!\right]$}%
}%
\ifdim\wd\matricesbox>\halfwidth\myboxwidth=\hsize\else\myboxwidth=\halfwidth\fi
\vbox{%
\ifdim\myboxwidth=\hsize
\setbox\onelinebox=\hbox{%
\vbox{\hbox{%
$\Pi_{7,20}$ spans $L_{10.1}$%
}\hbox{%
$22\infty222\infty$%
}%
}%
\hfill\copy\matricesbox
}%
\ifdim\wd\onelinebox>\myboxwidth
\hbox to \myboxwidth{%
$\Pi_{7,20}$ spans $L_{10.1}$%
\hfil
$22\infty222\infty$%
}%
\box\matricesbox
\else
\hbox to \myboxwidth{%
\unhbox\onelinebox
}%
\fi
\else
\hbox to \myboxwidth{%
$\Pi_{7,20}$ spans $L_{10.1}$%
\hfil}%
\hbox to \myboxwidth{%
$22\infty222\infty$%
\hfil}%
\box\matricesbox
\fi
}%
\hfill\discretionary{}{}{}%
\setbox\matricesbox=\hbox{%
{$\left[\!\llap{\phantom{%
\begingroup \smaller\smaller\smaller\begin{tabular}{@{}c@{}}%
\phantom{0}\\\phantom{0}\\\phantom{0}
\end{tabular}\endgroup%
}}\right.$}%
\begingroup \smaller\smaller\smaller\begin{tabular}{@{}c@{}}%
-1/8\\\phantom{0}\\\phantom{0}
\end{tabular}\endgroup%
\kern3pt%
\begingroup \smaller\smaller\smaller\begin{tabular}{@{}c@{}}%
\phantom{0}\\21/2\\\phantom{0}
\end{tabular}\endgroup%
\kern3pt%
\begingroup \smaller\smaller\smaller\begin{tabular}{@{}c@{}}%
\phantom{0}\\\phantom{0}\\28
\end{tabular}\endgroup%
{$\left.\llap{\phantom{%
\begingroup \smaller\smaller\smaller\begin{tabular}{@{}c@{}}%
\phantom{0}\\\phantom{0}\\\phantom{0}
\end{tabular}\endgroup%
}}\!\right]$}%
{$\left[\!\llap{\phantom{%
\begingroup \smaller\smaller\smaller\begin{tabular}{@{}c@{}}%
0\\0\\0
\end{tabular}\endgroup%
}}\right.$}%
\begingroup \smaller\smaller\smaller\begin{tabular}{@{}c@{}}%
6\\-1\\0
\end{tabular}\endgroup%
\kern3pt%
\begingroup \smaller\smaller\smaller\begin{tabular}{@{}c@{}}%
112\\-8\\-6
\end{tabular}\endgroup%
\kern3pt%
\begingroup \smaller\smaller\smaller\begin{tabular}{@{}c@{}}%
42\\-1\\-3
\end{tabular}\endgroup%
\kern3pt%
\begingroup \smaller\smaller\smaller\begin{tabular}{@{}c@{}}%
14\\1\\-1
\end{tabular}\endgroup%
\kern3pt%
\begingroup \smaller\smaller\smaller\begin{tabular}{@{}c@{}}%
42\\5\\0
\end{tabular}\endgroup%
\kern3pt%
\begingroup \smaller\smaller\smaller\begin{tabular}{@{}c@{}}%
14\\1\\1
\end{tabular}\endgroup%
\kern3pt%
\begingroup \smaller\smaller\smaller\begin{tabular}{@{}c@{}}%
14\\-1\\1
\end{tabular}\endgroup%
{$\left.\llap{\phantom{%
\begingroup \smaller\smaller\smaller\begin{tabular}{@{}c@{}}%
0\\0\\0
\end{tabular}\endgroup%
}}\!\right]$}%
}%
\ifdim\wd\matricesbox>\halfwidth\myboxwidth=\hsize\else\myboxwidth=\halfwidth\fi
\vbox{%
\ifdim\myboxwidth=\hsize
\setbox\onelinebox=\hbox{%
\vbox{\hbox{%
$\Pi_{7,21}$ spans $L_{22.8}$%
}\hbox{%
$2226632$%
}%
}%
\hfill\copy\matricesbox
}%
\ifdim\wd\onelinebox>\myboxwidth
\hbox to \myboxwidth{%
$\Pi_{7,21}$ spans $L_{22.8}$%
\hfil
$2226632$%
}%
\box\matricesbox
\else
\hbox to \myboxwidth{%
\unhbox\onelinebox
}%
\fi
\else
\hbox to \myboxwidth{%
$\Pi_{7,21}$ spans $L_{22.8}$%
\hfil}%
\hbox to \myboxwidth{%
$2226632$%
\hfil}%
\box\matricesbox
\fi
}%
\hfill\discretionary{}{}{}%

\vskip2pt\hrule\vskip2pt

\leavevmode\setbox\matricesbox=\hbox{%
{$\left[\!\llap{\phantom{%
\begingroup \smaller\smaller\smaller\begin{tabular}{@{}c@{}}%
\phantom{0}\\\phantom{0}\\\phantom{0}
\end{tabular}\endgroup%
}}\right.$}%
\begingroup \smaller\smaller\smaller\begin{tabular}{@{}c@{}}%
-1/4\\\phantom{0}\\\phantom{0}
\end{tabular}\endgroup%
\kern3pt%
\begingroup \smaller\smaller\smaller\begin{tabular}{@{}c@{}}%
\phantom{0}\\15\\\phantom{0}
\end{tabular}\endgroup%
\kern3pt%
\begingroup \smaller\smaller\smaller\begin{tabular}{@{}c@{}}%
\phantom{0}\\\phantom{0}\\15
\end{tabular}\endgroup%
{$\left.\llap{\phantom{%
\begingroup \smaller\smaller\smaller\begin{tabular}{@{}c@{}}%
\phantom{0}\\\phantom{0}\\\phantom{0}
\end{tabular}\endgroup%
}}\!\right]$}%
{$\left[\!\llap{\phantom{%
\begingroup \smaller\smaller\smaller\begin{tabular}{@{}c@{}}%
0\\0\\0
\end{tabular}\endgroup%
}}\right.$}%
\begingroup \smaller\smaller\smaller\begin{tabular}{@{}c@{}}%
6\\-1\\0
\end{tabular}\endgroup%
\kern3pt%
\begingroup \smaller\smaller\smaller\begin{tabular}{@{}c@{}}%
20\\-2\\-2
\end{tabular}\endgroup%
{$\left.\llap{\phantom{%
\begingroup \smaller\smaller\smaller\begin{tabular}{@{}c@{}}%
0\\0\\0
\end{tabular}\endgroup%
}}\!\right]$}%
}%
\ifdim\wd\matricesbox>\halfwidth\myboxwidth=\hsize\else\myboxwidth=\halfwidth\fi
\vbox{%
\ifdim\myboxwidth=\hsize
\setbox\onelinebox=\hbox{%
\vbox{\hbox{%
$\Pi_{8,1}$ spans $L_{19.10}$%
}\hbox{%
$|2|2|2|2|2|2|2|2\rtimes D_{8}$%
}%
}%
\hfill\copy\matricesbox
}%
\ifdim\wd\onelinebox>\myboxwidth
\hbox to \myboxwidth{%
$\Pi_{8,1}$ spans $L_{19.10}$%
\hfil
$|2|2|2|2|2|2|2|2\rtimes D_{8}$%
}%
\box\matricesbox
\else
\hbox to \myboxwidth{%
\unhbox\onelinebox
}%
\fi
\else
\hbox to \myboxwidth{%
$\Pi_{8,1}$ spans $L_{19.10}$%
\hfil}%
\hbox to \myboxwidth{%
$|2|2|2|2|2|2|2|2\rtimes D_{8}$%
\hfil}%
\box\matricesbox
\fi
}%
\hfill\discretionary{}{}{}%
\setbox\matricesbox=\hbox{%
{$\left[\!\llap{\phantom{%
\begingroup \smaller\smaller\smaller\begin{tabular}{@{}c@{}}%
\phantom{0}\\\phantom{0}\\\phantom{0}
\end{tabular}\endgroup%
}}\right.$}%
\begingroup \smaller\smaller\smaller\begin{tabular}{@{}c@{}}%
-1\\\phantom{0}\\\phantom{0}
\end{tabular}\endgroup%
\kern3pt%
\begingroup \smaller\smaller\smaller\begin{tabular}{@{}c@{}}%
\phantom{0}\\6\\\phantom{0}
\end{tabular}\endgroup%
\kern3pt%
\begingroup \smaller\smaller\smaller\begin{tabular}{@{}c@{}}%
\phantom{0}\\\phantom{0}\\6
\end{tabular}\endgroup%
{$\left.\llap{\phantom{%
\begingroup \smaller\smaller\smaller\begin{tabular}{@{}c@{}}%
\phantom{0}\\\phantom{0}\\\phantom{0}
\end{tabular}\endgroup%
}}\!\right]$}%
{$\left[\!\llap{\phantom{%
\begingroup \smaller\smaller\smaller\begin{tabular}{@{}c@{}}%
0\\0\\0
\end{tabular}\endgroup%
}}\right.$}%
\begingroup \smaller\smaller\smaller\begin{tabular}{@{}c@{}}%
2\\0\\-1
\end{tabular}\endgroup%
\kern3pt%
\begingroup \smaller\smaller\smaller\begin{tabular}{@{}c@{}}%
3\\1\\-1
\end{tabular}\endgroup%
{$\left.\llap{\phantom{%
\begingroup \smaller\smaller\smaller\begin{tabular}{@{}c@{}}%
0\\0\\0
\end{tabular}\endgroup%
}}\!\right]$}%
}%
\ifdim\wd\matricesbox>\halfwidth\myboxwidth=\hsize\else\myboxwidth=\halfwidth\fi
\vbox{%
\ifdim\myboxwidth=\hsize
\setbox\onelinebox=\hbox{%
\vbox{\hbox{%
$\Pi_{8,2}$ spans $L_{123.8}$%
}\hbox{%
$|2|2|2|2|2|2|2|2\rtimes D_{8}$%
}%
}%
\hfill\copy\matricesbox
}%
\ifdim\wd\onelinebox>\myboxwidth
\hbox to \myboxwidth{%
$\Pi_{8,2}$ spans $L_{123.8}$%
\hfil
$|2|2|2|2|2|2|2|2\rtimes D_{8}$%
}%
\box\matricesbox
\else
\hbox to \myboxwidth{%
\unhbox\onelinebox
}%
\fi
\else
\hbox to \myboxwidth{%
$\Pi_{8,2}$ spans $L_{123.8}$%
\hfil}%
\hbox to \myboxwidth{%
$|2|2|2|2|2|2|2|2\rtimes D_{8}$%
\hfil}%
\box\matricesbox
\fi
}%
\hfill\discretionary{}{}{}%
\setbox\matricesbox=\hbox{%
{$\left[\!\llap{\phantom{%
\begingroup \smaller\smaller\smaller\begin{tabular}{@{}c@{}}%
\phantom{0}\\\phantom{0}\\\phantom{0}
\end{tabular}\endgroup%
}}\right.$}%
\begingroup \smaller\smaller\smaller\begin{tabular}{@{}c@{}}%
-7/4\\\phantom{0}\\\phantom{0}
\end{tabular}\endgroup%
\kern3pt%
\begingroup \smaller\smaller\smaller\begin{tabular}{@{}c@{}}%
\phantom{0}\\1\\\phantom{0}
\end{tabular}\endgroup%
\kern3pt%
\begingroup \smaller\smaller\smaller\begin{tabular}{@{}c@{}}%
\phantom{0}\\\phantom{0}\\1
\end{tabular}\endgroup%
{$\left.\llap{\phantom{%
\begingroup \smaller\smaller\smaller\begin{tabular}{@{}c@{}}%
\phantom{0}\\\phantom{0}\\\phantom{0}
\end{tabular}\endgroup%
}}\!\right]$}%
{$\left[\!\llap{\phantom{%
\begingroup \smaller\smaller\smaller\begin{tabular}{@{}c@{}}%
0\\0\\0
\end{tabular}\endgroup%
}}\right.$}%
\begingroup \smaller\smaller\smaller\begin{tabular}{@{}c@{}}%
2\\0\\-3
\end{tabular}\endgroup%
\kern3pt%
\begingroup \smaller\smaller\smaller\begin{tabular}{@{}c@{}}%
4\\4\\-4
\end{tabular}\endgroup%
{$\left.\llap{\phantom{%
\begingroup \smaller\smaller\smaller\begin{tabular}{@{}c@{}}%
0\\0\\0
\end{tabular}\endgroup%
}}\!\right]$}%
}%
\ifdim\wd\matricesbox>\halfwidth\myboxwidth=\hsize\else\myboxwidth=\halfwidth\fi
\vbox{%
\ifdim\myboxwidth=\hsize
\setbox\onelinebox=\hbox{%
\vbox{\hbox{%
$\Pi_{8,3}$ spans $L_{8.1}$%
}\hbox{%
$|4|4|4|4|4|4|4|4\rtimes D_{8}$%
}%
}%
\hfill\copy\matricesbox
}%
\ifdim\wd\onelinebox>\myboxwidth
\hbox to \myboxwidth{%
$\Pi_{8,3}$ spans $L_{8.1}$%
\hfil
$|4|4|4|4|4|4|4|4\rtimes D_{8}$%
}%
\box\matricesbox
\else
\hbox to \myboxwidth{%
\unhbox\onelinebox
}%
\fi
\else
\hbox to \myboxwidth{%
$\Pi_{8,3}$ spans $L_{8.1}$%
\hfil}%
\hbox to \myboxwidth{%
$|4|4|4|4|4|4|4|4\rtimes D_{8}$%
\hfil}%
\box\matricesbox
\fi
}%
\hfill\discretionary{}{}{}%
\setbox\matricesbox=\hbox{%
{$\left[\!\llap{\phantom{%
\begingroup \smaller\smaller\smaller\begin{tabular}{@{}c@{}}%
\phantom{0}\\\phantom{0}\\\phantom{0}
\end{tabular}\endgroup%
}}\right.$}%
\begingroup \smaller\smaller\smaller\begin{tabular}{@{}c@{}}%
-4\\\phantom{0}\\\phantom{0}
\end{tabular}\endgroup%
\kern3pt%
\begingroup \smaller\smaller\smaller\begin{tabular}{@{}c@{}}%
\phantom{0}\\1\\\phantom{0}
\end{tabular}\endgroup%
\kern3pt%
\begingroup \smaller\smaller\smaller\begin{tabular}{@{}c@{}}%
\phantom{0}\\\phantom{0}\\1
\end{tabular}\endgroup%
{$\left.\llap{\phantom{%
\begingroup \smaller\smaller\smaller\begin{tabular}{@{}c@{}}%
\phantom{0}\\\phantom{0}\\\phantom{0}
\end{tabular}\endgroup%
}}\!\right]$}%
{$\left[\!\llap{\phantom{%
\begingroup \smaller\smaller\smaller\begin{tabular}{@{}c@{}}%
0\\0\\0
\end{tabular}\endgroup%
}}\right.$}%
\begingroup \smaller\smaller\smaller\begin{tabular}{@{}c@{}}%
1\\-1\\2
\end{tabular}\endgroup%
{$\left.\llap{\phantom{%
\begingroup \smaller\smaller\smaller\begin{tabular}{@{}c@{}}%
0\\0\\0
\end{tabular}\endgroup%
}}\!\right]$}%
}%
\ifdim\wd\matricesbox>\halfwidth\myboxwidth=\hsize\else\myboxwidth=\halfwidth\fi
\vbox{%
\ifdim\myboxwidth=\hsize
\setbox\onelinebox=\hbox{%
\vbox{\hbox{%
$\Pi_{8,4}=A_{2,III}=\hbox{GN}_{59}$ spans $L_{145.1}$%
}\hbox{%
$\slashinfty\slashtwo\slashinfty\slashtwo\slashinfty\slashtwo\slashinfty\slashtwo\rtimes D_{8}$%
}%
}%
\hfill\copy\matricesbox
}%
\ifdim\wd\onelinebox>\myboxwidth
\hbox to \myboxwidth{%
$\Pi_{8,4}=A_{2,III}=\hbox{GN}_{59}$ spans $L_{145.1}$%
\hfil
$\slashinfty\slashtwo\slashinfty\slashtwo\slashinfty\slashtwo\slashinfty\slashtwo\rtimes D_{8}$%
}%
\box\matricesbox
\else
\hbox to \myboxwidth{%
\unhbox\onelinebox
}%
\fi
\else
\hbox to \myboxwidth{%
$\Pi_{8,4}=A_{2,III}=\hbox{GN}_{59}$ spans $L_{145.1}$%
\hfil}%
\hbox to \myboxwidth{%
$\slashinfty\slashtwo\slashinfty\slashtwo\slashinfty\slashtwo\slashinfty\slashtwo\rtimes D_{8}$%
\hfil}%
\box\matricesbox
\fi
}%
\hfill\discretionary{}{}{}%
\setbox\matricesbox=\hbox{%
{$\left[\!\llap{\phantom{%
\begingroup \smaller\smaller\smaller\begin{tabular}{@{}c@{}}%
\phantom{0}\\\phantom{0}\\\phantom{0}
\end{tabular}\endgroup%
}}\right.$}%
\begingroup \smaller\smaller\smaller\begin{tabular}{@{}c@{}}%
-1/4\\\phantom{0}\\\phantom{0}
\end{tabular}\endgroup%
\kern3pt%
\begingroup \smaller\smaller\smaller\begin{tabular}{@{}c@{}}%
\phantom{0}\\3\\\phantom{0}
\end{tabular}\endgroup%
\kern3pt%
\begingroup \smaller\smaller\smaller\begin{tabular}{@{}c@{}}%
\phantom{0}\\\phantom{0}\\35
\end{tabular}\endgroup%
{$\left.\llap{\phantom{%
\begingroup \smaller\smaller\smaller\begin{tabular}{@{}c@{}}%
\phantom{0}\\\phantom{0}\\\phantom{0}
\end{tabular}\endgroup%
}}\!\right]$}%
{$\left[\!\llap{\phantom{%
\begingroup \smaller\smaller\smaller\begin{tabular}{@{}c@{}}%
0\\0\\0
\end{tabular}\endgroup%
}}\right.$}%
\begingroup \smaller\smaller\smaller\begin{tabular}{@{}c@{}}%
2\\1\\0
\end{tabular}\endgroup%
\kern3pt%
\begingroup \smaller\smaller\smaller\begin{tabular}{@{}c@{}}%
84\\14\\-6
\end{tabular}\endgroup%
\kern3pt%
\begingroup \smaller\smaller\smaller\begin{tabular}{@{}c@{}}%
10\\0\\-1
\end{tabular}\endgroup%
{$\left.\llap{\phantom{%
\begingroup \smaller\smaller\smaller\begin{tabular}{@{}c@{}}%
0\\0\\0
\end{tabular}\endgroup%
}}\!\right]$}%
}%
\ifdim\wd\matricesbox>\halfwidth\myboxwidth=\hsize\else\myboxwidth=\halfwidth\fi
\vbox{%
\ifdim\myboxwidth=\hsize
\setbox\onelinebox=\hbox{%
\vbox{\hbox{%
$\Pi_{8,5}$ spans $L_{69.1}$%
}\hbox{%
$|22|22|22|22\rtimes D_{4}$%
}%
}%
\hfill\copy\matricesbox
}%
\ifdim\wd\onelinebox>\myboxwidth
\hbox to \myboxwidth{%
$\Pi_{8,5}$ spans $L_{69.1}$%
\hfil
$|22|22|22|22\rtimes D_{4}$%
}%
\box\matricesbox
\else
\hbox to \myboxwidth{%
\unhbox\onelinebox
}%
\fi
\else
\hbox to \myboxwidth{%
$\Pi_{8,5}$ spans $L_{69.1}$%
\hfil}%
\hbox to \myboxwidth{%
$|22|22|22|22\rtimes D_{4}$%
\hfil}%
\box\matricesbox
\fi
}%
\hfill\discretionary{}{}{}%
\setbox\matricesbox=\hbox{%
{$\left[\!\llap{\phantom{%
\begingroup \smaller\smaller\smaller\begin{tabular}{@{}c@{}}%
\phantom{0}\\\phantom{0}\\\phantom{0}
\end{tabular}\endgroup%
}}\right.$}%
\begingroup \smaller\smaller\smaller\begin{tabular}{@{}c@{}}%
-1/2\\\phantom{0}\\\phantom{0}
\end{tabular}\endgroup%
\kern3pt%
\begingroup \smaller\smaller\smaller\begin{tabular}{@{}c@{}}%
\phantom{0}\\3/2\\\phantom{0}
\end{tabular}\endgroup%
\kern3pt%
\begingroup \smaller\smaller\smaller\begin{tabular}{@{}c@{}}%
\phantom{0}\\\phantom{0}\\24
\end{tabular}\endgroup%
{$\left.\llap{\phantom{%
\begingroup \smaller\smaller\smaller\begin{tabular}{@{}c@{}}%
\phantom{0}\\\phantom{0}\\\phantom{0}
\end{tabular}\endgroup%
}}\!\right]$}%
{$\left[\!\llap{\phantom{%
\begingroup \smaller\smaller\smaller\begin{tabular}{@{}c@{}}%
0\\0\\0
\end{tabular}\endgroup%
}}\right.$}%
\begingroup \smaller\smaller\smaller\begin{tabular}{@{}c@{}}%
1\\-1\\0
\end{tabular}\endgroup%
\kern3pt%
\begingroup \smaller\smaller\smaller\begin{tabular}{@{}c@{}}%
24\\-8\\-3
\end{tabular}\endgroup%
\kern3pt%
\begingroup \smaller\smaller\smaller\begin{tabular}{@{}c@{}}%
6\\0\\-1
\end{tabular}\endgroup%
{$\left.\llap{\phantom{%
\begingroup \smaller\smaller\smaller\begin{tabular}{@{}c@{}}%
0\\0\\0
\end{tabular}\endgroup%
}}\!\right]$}%
}%
\ifdim\wd\matricesbox>\halfwidth\myboxwidth=\hsize\else\myboxwidth=\halfwidth\fi
\vbox{%
\ifdim\myboxwidth=\hsize
\setbox\onelinebox=\hbox{%
\vbox{\hbox{%
$\Pi_{8,6}$ spans $L_{123.11}$%
}\hbox{%
$|22|22|22|22\rtimes D_{4}$%
}%
}%
\hfill\copy\matricesbox
}%
\ifdim\wd\onelinebox>\myboxwidth
\hbox to \myboxwidth{%
$\Pi_{8,6}$ spans $L_{123.11}$%
\hfil
$|22|22|22|22\rtimes D_{4}$%
}%
\box\matricesbox
\else
\hbox to \myboxwidth{%
\unhbox\onelinebox
}%
\fi
\else
\hbox to \myboxwidth{%
$\Pi_{8,6}$ spans $L_{123.11}$%
\hfil}%
\hbox to \myboxwidth{%
$|22|22|22|22\rtimes D_{4}$%
\hfil}%
\box\matricesbox
\fi
}%
\hfill\discretionary{}{}{}%
\setbox\matricesbox=\hbox{%
{$\left[\!\llap{\phantom{%
\begingroup \smaller\smaller\smaller\begin{tabular}{@{}c@{}}%
\phantom{0}\\\phantom{0}\\\phantom{0}
\end{tabular}\endgroup%
}}\right.$}%
\begingroup \smaller\smaller\smaller\begin{tabular}{@{}c@{}}%
-1/4\\\phantom{0}\\\phantom{0}
\end{tabular}\endgroup%
\kern3pt%
\begingroup \smaller\smaller\smaller\begin{tabular}{@{}c@{}}%
\phantom{0}\\3\\\phantom{0}
\end{tabular}\endgroup%
\kern3pt%
\begingroup \smaller\smaller\smaller\begin{tabular}{@{}c@{}}%
\phantom{0}\\\phantom{0}\\63
\end{tabular}\endgroup%
{$\left.\llap{\phantom{%
\begingroup \smaller\smaller\smaller\begin{tabular}{@{}c@{}}%
\phantom{0}\\\phantom{0}\\\phantom{0}
\end{tabular}\endgroup%
}}\!\right]$}%
{$\left[\!\llap{\phantom{%
\begingroup \smaller\smaller\smaller\begin{tabular}{@{}c@{}}%
0\\0\\0
\end{tabular}\endgroup%
}}\right.$}%
\begingroup \smaller\smaller\smaller\begin{tabular}{@{}c@{}}%
2\\-1\\0
\end{tabular}\endgroup%
\kern3pt%
\begingroup \smaller\smaller\smaller\begin{tabular}{@{}c@{}}%
36\\-6\\-2
\end{tabular}\endgroup%
\kern3pt%
\begingroup \smaller\smaller\smaller\begin{tabular}{@{}c@{}}%
14\\0\\-1
\end{tabular}\endgroup%
{$\left.\llap{\phantom{%
\begingroup \smaller\smaller\smaller\begin{tabular}{@{}c@{}}%
0\\0\\0
\end{tabular}\endgroup%
}}\!\right]$}%
}%
\ifdim\wd\matricesbox>\halfwidth\myboxwidth=\hsize\else\myboxwidth=\halfwidth\fi
\vbox{%
\ifdim\myboxwidth=\hsize
\setbox\onelinebox=\hbox{%
\vbox{\hbox{%
$\Pi_{8,7}$ spans $L_{25.1}$%
}\hbox{%
$|22|22|22|22\rtimes D_{4}$%
}%
}%
\hfill\copy\matricesbox
}%
\ifdim\wd\onelinebox>\myboxwidth
\hbox to \myboxwidth{%
$\Pi_{8,7}$ spans $L_{25.1}$%
\hfil
$|22|22|22|22\rtimes D_{4}$%
}%
\box\matricesbox
\else
\hbox to \myboxwidth{%
\unhbox\onelinebox
}%
\fi
\else
\hbox to \myboxwidth{%
$\Pi_{8,7}$ spans $L_{25.1}$%
\hfil}%
\hbox to \myboxwidth{%
$|22|22|22|22\rtimes D_{4}$%
\hfil}%
\box\matricesbox
\fi
}%
\hfill\discretionary{}{}{}%
\setbox\matricesbox=\hbox{%
{$\left[\!\llap{\phantom{%
\begingroup \smaller\smaller\smaller\begin{tabular}{@{}c@{}}%
\phantom{0}\\\phantom{0}\\\phantom{0}
\end{tabular}\endgroup%
}}\right.$}%
\begingroup \smaller\smaller\smaller\begin{tabular}{@{}c@{}}%
-1/2\\\phantom{0}\\\phantom{0}
\end{tabular}\endgroup%
\kern3pt%
\begingroup \smaller\smaller\smaller\begin{tabular}{@{}c@{}}%
\phantom{0}\\3/2\\\phantom{0}
\end{tabular}\endgroup%
\kern3pt%
\begingroup \smaller\smaller\smaller\begin{tabular}{@{}c@{}}%
\phantom{0}\\\phantom{0}\\10
\end{tabular}\endgroup%
{$\left.\llap{\phantom{%
\begingroup \smaller\smaller\smaller\begin{tabular}{@{}c@{}}%
\phantom{0}\\\phantom{0}\\\phantom{0}
\end{tabular}\endgroup%
}}\!\right]$}%
{$\left[\!\llap{\phantom{%
\begingroup \smaller\smaller\smaller\begin{tabular}{@{}c@{}}%
0\\0\\0
\end{tabular}\endgroup%
}}\right.$}%
\begingroup \smaller\smaller\smaller\begin{tabular}{@{}c@{}}%
1\\1\\0
\end{tabular}\endgroup%
\kern3pt%
\begingroup \smaller\smaller\smaller\begin{tabular}{@{}c@{}}%
15\\5\\-3
\end{tabular}\endgroup%
\kern3pt%
\begingroup \smaller\smaller\smaller\begin{tabular}{@{}c@{}}%
8\\0\\-2
\end{tabular}\endgroup%
{$\left.\llap{\phantom{%
\begingroup \smaller\smaller\smaller\begin{tabular}{@{}c@{}}%
0\\0\\0
\end{tabular}\endgroup%
}}\!\right]$}%
}%
\ifdim\wd\matricesbox>\halfwidth\myboxwidth=\hsize\else\myboxwidth=\halfwidth\fi
\vbox{%
\ifdim\myboxwidth=\hsize
\setbox\onelinebox=\hbox{%
\vbox{\hbox{%
$\Pi_{8,8}$ spans $L_{17.5}$%
}\hbox{%
$|22|22|22|22\rtimes D_{4}$%
}%
}%
\hfill\copy\matricesbox
}%
\ifdim\wd\onelinebox>\myboxwidth
\hbox to \myboxwidth{%
$\Pi_{8,8}$ spans $L_{17.5}$%
\hfil
$|22|22|22|22\rtimes D_{4}$%
}%
\box\matricesbox
\else
\hbox to \myboxwidth{%
\unhbox\onelinebox
}%
\fi
\else
\hbox to \myboxwidth{%
$\Pi_{8,8}$ spans $L_{17.5}$%
\hfil}%
\hbox to \myboxwidth{%
$|22|22|22|22\rtimes D_{4}$%
\hfil}%
\box\matricesbox
\fi
}%
\hfill\discretionary{}{}{}%
\setbox\matricesbox=\hbox{%
{$\left[\!\llap{\phantom{%
\begingroup \smaller\smaller\smaller\begin{tabular}{@{}c@{}}%
\phantom{0}\\\phantom{0}\\\phantom{0}
\end{tabular}\endgroup%
}}\right.$}%
\begingroup \smaller\smaller\smaller\begin{tabular}{@{}c@{}}%
-1/2\\\phantom{0}\\\phantom{0}
\end{tabular}\endgroup%
\kern3pt%
\begingroup \smaller\smaller\smaller\begin{tabular}{@{}c@{}}%
\phantom{0}\\3/2\\\phantom{0}
\end{tabular}\endgroup%
\kern3pt%
\begingroup \smaller\smaller\smaller\begin{tabular}{@{}c@{}}%
\phantom{0}\\\phantom{0}\\60
\end{tabular}\endgroup%
{$\left.\llap{\phantom{%
\begingroup \smaller\smaller\smaller\begin{tabular}{@{}c@{}}%
\phantom{0}\\\phantom{0}\\\phantom{0}
\end{tabular}\endgroup%
}}\!\right]$}%
{$\left[\!\llap{\phantom{%
\begingroup \smaller\smaller\smaller\begin{tabular}{@{}c@{}}%
0\\0\\0
\end{tabular}\endgroup%
}}\right.$}%
\begingroup \smaller\smaller\smaller\begin{tabular}{@{}c@{}}%
1\\-1\\0
\end{tabular}\endgroup%
\kern3pt%
\begingroup \smaller\smaller\smaller\begin{tabular}{@{}c@{}}%
12\\-4\\-1
\end{tabular}\endgroup%
\kern3pt%
\begingroup \smaller\smaller\smaller\begin{tabular}{@{}c@{}}%
10\\0\\-1
\end{tabular}\endgroup%
{$\left.\llap{\phantom{%
\begingroup \smaller\smaller\smaller\begin{tabular}{@{}c@{}}%
0\\0\\0
\end{tabular}\endgroup%
}}\!\right]$}%
}%
\ifdim\wd\matricesbox>\halfwidth\myboxwidth=\hsize\else\myboxwidth=\halfwidth\fi
\vbox{%
\ifdim\myboxwidth=\hsize
\setbox\onelinebox=\hbox{%
\vbox{\hbox{%
$\Pi_{8,9}$ spans $L_{17.12}$%
}\hbox{%
$|22|22|22|22\rtimes D_{4}$%
}%
}%
\hfill\copy\matricesbox
}%
\ifdim\wd\onelinebox>\myboxwidth
\hbox to \myboxwidth{%
$\Pi_{8,9}$ spans $L_{17.12}$%
\hfil
$|22|22|22|22\rtimes D_{4}$%
}%
\box\matricesbox
\else
\hbox to \myboxwidth{%
\unhbox\onelinebox
}%
\fi
\else
\hbox to \myboxwidth{%
$\Pi_{8,9}$ spans $L_{17.12}$%
\hfil}%
\hbox to \myboxwidth{%
$|22|22|22|22\rtimes D_{4}$%
\hfil}%
\box\matricesbox
\fi
}%
\hfill\discretionary{}{}{}%
\setbox\matricesbox=\hbox{%
{$\left[\!\llap{\phantom{%
\begingroup \smaller\smaller\smaller\begin{tabular}{@{}c@{}}%
\phantom{0}\\\phantom{0}\\\phantom{0}
\end{tabular}\endgroup%
}}\right.$}%
\begingroup \smaller\smaller\smaller\begin{tabular}{@{}c@{}}%
-1/8\\\phantom{0}\\\phantom{0}
\end{tabular}\endgroup%
\kern3pt%
\begingroup \smaller\smaller\smaller\begin{tabular}{@{}c@{}}%
\phantom{0}\\3/2\\\phantom{0}
\end{tabular}\endgroup%
\kern3pt%
\begingroup \smaller\smaller\smaller\begin{tabular}{@{}c@{}}%
\phantom{0}\\\phantom{0}\\21
\end{tabular}\endgroup%
{$\left.\llap{\phantom{%
\begingroup \smaller\smaller\smaller\begin{tabular}{@{}c@{}}%
\phantom{0}\\\phantom{0}\\\phantom{0}
\end{tabular}\endgroup%
}}\!\right]$}%
{$\left[\!\llap{\phantom{%
\begingroup \smaller\smaller\smaller\begin{tabular}{@{}c@{}}%
0\\0\\0
\end{tabular}\endgroup%
}}\right.$}%
\begingroup \smaller\smaller\smaller\begin{tabular}{@{}c@{}}%
4\\2\\0
\end{tabular}\endgroup%
\kern3pt%
\begingroup \smaller\smaller\smaller\begin{tabular}{@{}c@{}}%
42\\7\\-3
\end{tabular}\endgroup%
\kern3pt%
\begingroup \smaller\smaller\smaller\begin{tabular}{@{}c@{}}%
48\\0\\-4
\end{tabular}\endgroup%
{$\left.\llap{\phantom{%
\begingroup \smaller\smaller\smaller\begin{tabular}{@{}c@{}}%
0\\0\\0
\end{tabular}\endgroup%
}}\!\right]$}%
}%
\ifdim\wd\matricesbox>\halfwidth\myboxwidth=\hsize\else\myboxwidth=\halfwidth\fi
\vbox{%
\ifdim\myboxwidth=\hsize
\setbox\onelinebox=\hbox{%
\vbox{\hbox{%
$\Pi_{8,10}$ spans $L_{25.8}$%
}\hbox{%
$|22|22|22|22\rtimes D_{4}$%
}%
}%
\hfill\copy\matricesbox
}%
\ifdim\wd\onelinebox>\myboxwidth
\hbox to \myboxwidth{%
$\Pi_{8,10}$ spans $L_{25.8}$%
\hfil
$|22|22|22|22\rtimes D_{4}$%
}%
\box\matricesbox
\else
\hbox to \myboxwidth{%
\unhbox\onelinebox
}%
\fi
\else
\hbox to \myboxwidth{%
$\Pi_{8,10}$ spans $L_{25.8}$%
\hfil}%
\hbox to \myboxwidth{%
$|22|22|22|22\rtimes D_{4}$%
\hfil}%
\box\matricesbox
\fi
}%
\hfill\discretionary{}{}{}%
\setbox\matricesbox=\hbox{%
{$\left[\!\llap{\phantom{%
\begingroup \smaller\smaller\smaller\begin{tabular}{@{}c@{}}%
\phantom{0}\\\phantom{0}\\\phantom{0}
\end{tabular}\endgroup%
}}\right.$}%
\begingroup \smaller\smaller\smaller\begin{tabular}{@{}c@{}}%
-1/2\\\phantom{0}\\\phantom{0}
\end{tabular}\endgroup%
\kern3pt%
\begingroup \smaller\smaller\smaller\begin{tabular}{@{}c@{}}%
\phantom{0}\\3/2\\\phantom{0}
\end{tabular}\endgroup%
\kern3pt%
\begingroup \smaller\smaller\smaller\begin{tabular}{@{}c@{}}%
\phantom{0}\\\phantom{0}\\36
\end{tabular}\endgroup%
{$\left.\llap{\phantom{%
\begingroup \smaller\smaller\smaller\begin{tabular}{@{}c@{}}%
\phantom{0}\\\phantom{0}\\\phantom{0}
\end{tabular}\endgroup%
}}\!\right]$}%
{$\left[\!\llap{\phantom{%
\begingroup \smaller\smaller\smaller\begin{tabular}{@{}c@{}}%
0\\0\\0
\end{tabular}\endgroup%
}}\right.$}%
\begingroup \smaller\smaller\smaller\begin{tabular}{@{}c@{}}%
1\\-1\\0
\end{tabular}\endgroup%
\kern3pt%
\begingroup \smaller\smaller\smaller\begin{tabular}{@{}c@{}}%
9\\-3\\-1
\end{tabular}\endgroup%
\kern3pt%
\begingroup \smaller\smaller\smaller\begin{tabular}{@{}c@{}}%
16\\0\\-2
\end{tabular}\endgroup%
{$\left.\llap{\phantom{%
\begingroup \smaller\smaller\smaller\begin{tabular}{@{}c@{}}%
0\\0\\0
\end{tabular}\endgroup%
}}\!\right]$}%
}%
\ifdim\wd\matricesbox>\halfwidth\myboxwidth=\hsize\else\myboxwidth=\halfwidth\fi
\vbox{%
\ifdim\myboxwidth=\hsize
\setbox\onelinebox=\hbox{%
\vbox{\hbox{%
$\Pi_{8,11}=\hbox{GN}_{32}$ spans $L_{172.5}$%
}\hbox{%
$|22|22|22|22\rtimes D_{4}$%
}%
}%
\hfill\copy\matricesbox
}%
\ifdim\wd\onelinebox>\myboxwidth
\hbox to \myboxwidth{%
$\Pi_{8,11}=\hbox{GN}_{32}$ spans $L_{172.5}$%
\hfil
$|22|22|22|22\rtimes D_{4}$%
}%
\box\matricesbox
\else
\hbox to \myboxwidth{%
\unhbox\onelinebox
}%
\fi
\else
\hbox to \myboxwidth{%
$\Pi_{8,11}=\hbox{GN}_{32}$ spans $L_{172.5}$%
\hfil}%
\hbox to \myboxwidth{%
$|22|22|22|22\rtimes D_{4}$%
\hfil}%
\box\matricesbox
\fi
}%
\hfill\discretionary{}{}{}%
\setbox\matricesbox=\hbox{%
{$\left[\!\llap{\phantom{%
\begingroup \smaller\smaller\smaller\begin{tabular}{@{}c@{}}%
\phantom{0}\\\phantom{0}\\\phantom{0}
\end{tabular}\endgroup%
}}\right.$}%
\begingroup \smaller\smaller\smaller\begin{tabular}{@{}c@{}}%
-1/8\\\phantom{0}\\\phantom{0}
\end{tabular}\endgroup%
\kern3pt%
\begingroup \smaller\smaller\smaller\begin{tabular}{@{}c@{}}%
\phantom{0}\\3/2\\\phantom{0}
\end{tabular}\endgroup%
\kern3pt%
\begingroup \smaller\smaller\smaller\begin{tabular}{@{}c@{}}%
\phantom{0}\\\phantom{0}\\105
\end{tabular}\endgroup%
{$\left.\llap{\phantom{%
\begingroup \smaller\smaller\smaller\begin{tabular}{@{}c@{}}%
\phantom{0}\\\phantom{0}\\\phantom{0}
\end{tabular}\endgroup%
}}\!\right]$}%
{$\left[\!\llap{\phantom{%
\begingroup \smaller\smaller\smaller\begin{tabular}{@{}c@{}}%
0\\0\\0
\end{tabular}\endgroup%
}}\right.$}%
\begingroup \smaller\smaller\smaller\begin{tabular}{@{}c@{}}%
4\\2\\0
\end{tabular}\endgroup%
\kern3pt%
\begingroup \smaller\smaller\smaller\begin{tabular}{@{}c@{}}%
30\\5\\-1
\end{tabular}\endgroup%
\kern3pt%
\begingroup \smaller\smaller\smaller\begin{tabular}{@{}c@{}}%
112\\0\\-4
\end{tabular}\endgroup%
{$\left.\llap{\phantom{%
\begingroup \smaller\smaller\smaller\begin{tabular}{@{}c@{}}%
0\\0\\0
\end{tabular}\endgroup%
}}\!\right]$}%
}%
\ifdim\wd\matricesbox>\halfwidth\myboxwidth=\hsize\else\myboxwidth=\halfwidth\fi
\vbox{%
\ifdim\myboxwidth=\hsize
\setbox\onelinebox=\hbox{%
\vbox{\hbox{%
$\Pi_{8,12}$ spans $L_{69.6}$%
}\hbox{%
$|22|22|22|22\rtimes D_{4}$%
}%
}%
\hfill\copy\matricesbox
}%
\ifdim\wd\onelinebox>\myboxwidth
\hbox to \myboxwidth{%
$\Pi_{8,12}$ spans $L_{69.6}$%
\hfil
$|22|22|22|22\rtimes D_{4}$%
}%
\box\matricesbox
\else
\hbox to \myboxwidth{%
\unhbox\onelinebox
}%
\fi
\else
\hbox to \myboxwidth{%
$\Pi_{8,12}$ spans $L_{69.6}$%
\hfil}%
\hbox to \myboxwidth{%
$|22|22|22|22\rtimes D_{4}$%
\hfil}%
\box\matricesbox
\fi
}%
\hfill\discretionary{}{}{}%
\setbox\matricesbox=\hbox{%
{$\left[\!\llap{\phantom{%
\begingroup \smaller\smaller\smaller\begin{tabular}{@{}c@{}}%
\phantom{0}\\\phantom{0}\\\phantom{0}
\end{tabular}\endgroup%
}}\right.$}%
\begingroup \smaller\smaller\smaller\begin{tabular}{@{}c@{}}%
-1/2\\\phantom{0}\\\phantom{0}
\end{tabular}\endgroup%
\kern3pt%
\begingroup \smaller\smaller\smaller\begin{tabular}{@{}c@{}}%
\phantom{0}\\4\\\phantom{0}
\end{tabular}\endgroup%
\kern3pt%
\begingroup \smaller\smaller\smaller\begin{tabular}{@{}c@{}}%
\phantom{0}\\\phantom{0}\\15/2
\end{tabular}\endgroup%
{$\left.\llap{\phantom{%
\begingroup \smaller\smaller\smaller\begin{tabular}{@{}c@{}}%
\phantom{0}\\\phantom{0}\\\phantom{0}
\end{tabular}\endgroup%
}}\!\right]$}%
{$\left[\!\llap{\phantom{%
\begingroup \smaller\smaller\smaller\begin{tabular}{@{}c@{}}%
0\\0\\0
\end{tabular}\endgroup%
}}\right.$}%
\begingroup \smaller\smaller\smaller\begin{tabular}{@{}c@{}}%
2\\1\\0
\end{tabular}\endgroup%
\kern3pt%
\begingroup \smaller\smaller\smaller\begin{tabular}{@{}c@{}}%
20\\5\\-4
\end{tabular}\endgroup%
\kern3pt%
\begingroup \smaller\smaller\smaller\begin{tabular}{@{}c@{}}%
3\\0\\-1
\end{tabular}\endgroup%
{$\left.\llap{\phantom{%
\begingroup \smaller\smaller\smaller\begin{tabular}{@{}c@{}}%
0\\0\\0
\end{tabular}\endgroup%
}}\!\right]$}%
}%
\ifdim\wd\matricesbox>\halfwidth\myboxwidth=\hsize\else\myboxwidth=\halfwidth\fi
\vbox{%
\ifdim\myboxwidth=\hsize
\setbox\onelinebox=\hbox{%
\vbox{\hbox{%
$\Pi_{8,13}$ spans $L_{17.6}$%
}\hbox{%
$|22|22|22|22\rtimes D_{4}$%
}%
}%
\hfill\copy\matricesbox
}%
\ifdim\wd\onelinebox>\myboxwidth
\hbox to \myboxwidth{%
$\Pi_{8,13}$ spans $L_{17.6}$%
\hfil
$|22|22|22|22\rtimes D_{4}$%
}%
\box\matricesbox
\else
\hbox to \myboxwidth{%
\unhbox\onelinebox
}%
\fi
\else
\hbox to \myboxwidth{%
$\Pi_{8,13}$ spans $L_{17.6}$%
\hfil}%
\hbox to \myboxwidth{%
$|22|22|22|22\rtimes D_{4}$%
\hfil}%
\box\matricesbox
\fi
}%
\hfill\discretionary{}{}{}%
\setbox\matricesbox=\hbox{%
{$\left[\!\llap{\phantom{%
\begingroup \smaller\smaller\smaller\begin{tabular}{@{}c@{}}%
\phantom{0}\\\phantom{0}\\\phantom{0}
\end{tabular}\endgroup%
}}\right.$}%
\begingroup \smaller\smaller\smaller\begin{tabular}{@{}c@{}}%
-1\\\phantom{0}\\\phantom{0}
\end{tabular}\endgroup%
\kern3pt%
\begingroup \smaller\smaller\smaller\begin{tabular}{@{}c@{}}%
\phantom{0}\\2\\\phantom{0}
\end{tabular}\endgroup%
\kern3pt%
\begingroup \smaller\smaller\smaller\begin{tabular}{@{}c@{}}%
\phantom{0}\\\phantom{0}\\6
\end{tabular}\endgroup%
{$\left.\llap{\phantom{%
\begingroup \smaller\smaller\smaller\begin{tabular}{@{}c@{}}%
\phantom{0}\\\phantom{0}\\\phantom{0}
\end{tabular}\endgroup%
}}\!\right]$}%
{$\left[\!\llap{\phantom{%
\begingroup \smaller\smaller\smaller\begin{tabular}{@{}c@{}}%
0\\0\\0
\end{tabular}\endgroup%
}}\right.$}%
\begingroup \smaller\smaller\smaller\begin{tabular}{@{}c@{}}%
1\\1\\0
\end{tabular}\endgroup%
\kern3pt%
\begingroup \smaller\smaller\smaller\begin{tabular}{@{}c@{}}%
6\\3\\-2
\end{tabular}\endgroup%
\kern3pt%
\begingroup \smaller\smaller\smaller\begin{tabular}{@{}c@{}}%
2\\0\\-1
\end{tabular}\endgroup%
{$\left.\llap{\phantom{%
\begingroup \smaller\smaller\smaller\begin{tabular}{@{}c@{}}%
0\\0\\0
\end{tabular}\endgroup%
}}\!\right]$}%
}%
\ifdim\wd\matricesbox>\halfwidth\myboxwidth=\hsize\else\myboxwidth=\halfwidth\fi
\vbox{%
\ifdim\myboxwidth=\hsize
\setbox\onelinebox=\hbox{%
\vbox{\hbox{%
$\Pi_{8,14}$ spans $L_{123.5}$%
}\hbox{%
$|22|22|22|22\rtimes D_{4}$%
}%
}%
\hfill\copy\matricesbox
}%
\ifdim\wd\onelinebox>\myboxwidth
\hbox to \myboxwidth{%
$\Pi_{8,14}$ spans $L_{123.5}$%
\hfil
$|22|22|22|22\rtimes D_{4}$%
}%
\box\matricesbox
\else
\hbox to \myboxwidth{%
\unhbox\onelinebox
}%
\fi
\else
\hbox to \myboxwidth{%
$\Pi_{8,14}$ spans $L_{123.5}$%
\hfil}%
\hbox to \myboxwidth{%
$|22|22|22|22\rtimes D_{4}$%
\hfil}%
\box\matricesbox
\fi
}%
\hfill\discretionary{}{}{}%
\setbox\matricesbox=\hbox{%
{$\left[\!\llap{\phantom{%
\begingroup \smaller\smaller\smaller\begin{tabular}{@{}c@{}}%
\phantom{0}\\\phantom{0}\\\phantom{0}
\end{tabular}\endgroup%
}}\right.$}%
\begingroup \smaller\smaller\smaller\begin{tabular}{@{}c@{}}%
-1\\\phantom{0}\\\phantom{0}
\end{tabular}\endgroup%
\kern3pt%
\begingroup \smaller\smaller\smaller\begin{tabular}{@{}c@{}}%
\phantom{0}\\2\\\phantom{0}
\end{tabular}\endgroup%
\kern3pt%
\begingroup \smaller\smaller\smaller\begin{tabular}{@{}c@{}}%
\phantom{0}\\\phantom{0}\\12
\end{tabular}\endgroup%
{$\left.\llap{\phantom{%
\begingroup \smaller\smaller\smaller\begin{tabular}{@{}c@{}}%
\phantom{0}\\\phantom{0}\\\phantom{0}
\end{tabular}\endgroup%
}}\!\right]$}%
{$\left[\!\llap{\phantom{%
\begingroup \smaller\smaller\smaller\begin{tabular}{@{}c@{}}%
0\\0\\0
\end{tabular}\endgroup%
}}\right.$}%
\begingroup \smaller\smaller\smaller\begin{tabular}{@{}c@{}}%
1\\1\\0
\end{tabular}\endgroup%
\kern3pt%
\begingroup \smaller\smaller\smaller\begin{tabular}{@{}c@{}}%
4\\2\\-1
\end{tabular}\endgroup%
\kern3pt%
\begingroup \smaller\smaller\smaller\begin{tabular}{@{}c@{}}%
3\\0\\-1
\end{tabular}\endgroup%
{$\left.\llap{\phantom{%
\begingroup \smaller\smaller\smaller\begin{tabular}{@{}c@{}}%
0\\0\\0
\end{tabular}\endgroup%
}}\!\right]$}%
}%
\ifdim\wd\matricesbox>\halfwidth\myboxwidth=\hsize\else\myboxwidth=\halfwidth\fi
\vbox{%
\ifdim\myboxwidth=\hsize
\setbox\onelinebox=\hbox{%
\vbox{\hbox{%
$\Pi_{8,15}$ spans $L_{151.6}$%
}\hbox{%
$|22|22|22|22\rtimes D_{4}$%
}%
}%
\hfill\copy\matricesbox
}%
\ifdim\wd\onelinebox>\myboxwidth
\hbox to \myboxwidth{%
$\Pi_{8,15}$ spans $L_{151.6}$%
\hfil
$|22|22|22|22\rtimes D_{4}$%
}%
\box\matricesbox
\else
\hbox to \myboxwidth{%
\unhbox\onelinebox
}%
\fi
\else
\hbox to \myboxwidth{%
$\Pi_{8,15}$ spans $L_{151.6}$%
\hfil}%
\hbox to \myboxwidth{%
$|22|22|22|22\rtimes D_{4}$%
\hfil}%
\box\matricesbox
\fi
}%
\hfill\discretionary{}{}{}%
\setbox\matricesbox=\hbox{%
{$\left[\!\llap{\phantom{%
\begingroup \smaller\smaller\smaller\begin{tabular}{@{}c@{}}%
\phantom{0}\\\phantom{0}\\\phantom{0}
\end{tabular}\endgroup%
}}\right.$}%
\begingroup \smaller\smaller\smaller\begin{tabular}{@{}c@{}}%
-1\\\phantom{0}\\\phantom{0}
\end{tabular}\endgroup%
\kern3pt%
\begingroup \smaller\smaller\smaller\begin{tabular}{@{}c@{}}%
\phantom{0}\\1/2\\\phantom{0}
\end{tabular}\endgroup%
\kern3pt%
\begingroup \smaller\smaller\smaller\begin{tabular}{@{}c@{}}%
\phantom{0}\\\phantom{0}\\15/2
\end{tabular}\endgroup%
{$\left.\llap{\phantom{%
\begingroup \smaller\smaller\smaller\begin{tabular}{@{}c@{}}%
\phantom{0}\\\phantom{0}\\\phantom{0}
\end{tabular}\endgroup%
}}\!\right]$}%
{$\left[\!\llap{\phantom{%
\begingroup \smaller\smaller\smaller\begin{tabular}{@{}c@{}}%
0\\0\\0
\end{tabular}\endgroup%
}}\right.$}%
\begingroup \smaller\smaller\smaller\begin{tabular}{@{}c@{}}%
1\\2\\0
\end{tabular}\endgroup%
\kern3pt%
\begingroup \smaller\smaller\smaller\begin{tabular}{@{}c@{}}%
3\\3\\1
\end{tabular}\endgroup%
\kern3pt%
\begingroup \smaller\smaller\smaller\begin{tabular}{@{}c@{}}%
5\\0\\2
\end{tabular}\endgroup%
{$\left.\llap{\phantom{%
\begingroup \smaller\smaller\smaller\begin{tabular}{@{}c@{}}%
0\\0\\0
\end{tabular}\endgroup%
}}\!\right]$}%
}%
\ifdim\wd\matricesbox>\halfwidth\myboxwidth=\hsize\else\myboxwidth=\halfwidth\fi
\vbox{%
\ifdim\myboxwidth=\hsize
\setbox\onelinebox=\hbox{%
\vbox{\hbox{%
$\Pi_{8,16}$ spans $L_{17.4}$%
}\hbox{%
$|22|22|22|22\rtimes D_{4}$%
}%
}%
\hfill\copy\matricesbox
}%
\ifdim\wd\onelinebox>\myboxwidth
\hbox to \myboxwidth{%
$\Pi_{8,16}$ spans $L_{17.4}$%
\hfil
$|22|22|22|22\rtimes D_{4}$%
}%
\box\matricesbox
\else
\hbox to \myboxwidth{%
\unhbox\onelinebox
}%
\fi
\else
\hbox to \myboxwidth{%
$\Pi_{8,16}$ spans $L_{17.4}$%
\hfil}%
\hbox to \myboxwidth{%
$|22|22|22|22\rtimes D_{4}$%
\hfil}%
\box\matricesbox
\fi
}%
\hfill\discretionary{}{}{}%
\setbox\matricesbox=\hbox{%
{$\left[\!\llap{\phantom{%
\begingroup \smaller\smaller\smaller\begin{tabular}{@{}c@{}}%
\phantom{0}\\\phantom{0}\\\phantom{0}
\end{tabular}\endgroup%
}}\right.$}%
\begingroup \smaller\smaller\smaller\begin{tabular}{@{}c@{}}%
-1/8\\\phantom{0}\\\phantom{0}
\end{tabular}\endgroup%
\kern3pt%
\begingroup \smaller\smaller\smaller\begin{tabular}{@{}c@{}}%
\phantom{0}\\21\\\phantom{0}
\end{tabular}\endgroup%
\kern3pt%
\begingroup \smaller\smaller\smaller\begin{tabular}{@{}c@{}}%
\phantom{0}\\\phantom{0}\\35/2
\end{tabular}\endgroup%
{$\left.\llap{\phantom{%
\begingroup \smaller\smaller\smaller\begin{tabular}{@{}c@{}}%
\phantom{0}\\\phantom{0}\\\phantom{0}
\end{tabular}\endgroup%
}}\!\right]$}%
{$\left[\!\llap{\phantom{%
\begingroup \smaller\smaller\smaller\begin{tabular}{@{}c@{}}%
0\\0\\0
\end{tabular}\endgroup%
}}\right.$}%
\begingroup \smaller\smaller\smaller\begin{tabular}{@{}c@{}}%
48\\-4\\0
\end{tabular}\endgroup%
\kern3pt%
\begingroup \smaller\smaller\smaller\begin{tabular}{@{}c@{}}%
14\\-1\\-1
\end{tabular}\endgroup%
\kern3pt%
\begingroup \smaller\smaller\smaller\begin{tabular}{@{}c@{}}%
20\\0\\-2
\end{tabular}\endgroup%
{$\left.\llap{\phantom{%
\begingroup \smaller\smaller\smaller\begin{tabular}{@{}c@{}}%
0\\0\\0
\end{tabular}\endgroup%
}}\!\right]$}%
}%
\ifdim\wd\matricesbox>\halfwidth\myboxwidth=\hsize\else\myboxwidth=\halfwidth\fi
\vbox{%
\ifdim\myboxwidth=\hsize
\setbox\onelinebox=\hbox{%
\vbox{\hbox{%
$\Pi_{8,17}$ spans $L_{69.11}$%
}\hbox{%
$2|22|22|22|2\rtimes D_{4}$%
}%
}%
\hfill\copy\matricesbox
}%
\ifdim\wd\onelinebox>\myboxwidth
\hbox to \myboxwidth{%
$\Pi_{8,17}$ spans $L_{69.11}$%
\hfil
$2|22|22|22|2\rtimes D_{4}$%
}%
\box\matricesbox
\else
\hbox to \myboxwidth{%
\unhbox\onelinebox
}%
\fi
\else
\hbox to \myboxwidth{%
$\Pi_{8,17}$ spans $L_{69.11}$%
\hfil}%
\hbox to \myboxwidth{%
$2|22|22|22|2\rtimes D_{4}$%
\hfil}%
\box\matricesbox
\fi
}%
\hfill\discretionary{}{}{}%
\setbox\matricesbox=\hbox{%
{$\left[\!\llap{\phantom{%
\begingroup \smaller\smaller\smaller\begin{tabular}{@{}c@{}}%
\phantom{0}\\\phantom{0}\\\phantom{0}
\end{tabular}\endgroup%
}}\right.$}%
\begingroup \smaller\smaller\smaller\begin{tabular}{@{}c@{}}%
-1/8\\\phantom{0}\\\phantom{0}
\end{tabular}\endgroup%
\kern3pt%
\begingroup \smaller\smaller\smaller\begin{tabular}{@{}c@{}}%
\phantom{0}\\15/2\\\phantom{0}
\end{tabular}\endgroup%
\kern3pt%
\begingroup \smaller\smaller\smaller\begin{tabular}{@{}c@{}}%
\phantom{0}\\\phantom{0}\\3
\end{tabular}\endgroup%
{$\left.\llap{\phantom{%
\begingroup \smaller\smaller\smaller\begin{tabular}{@{}c@{}}%
\phantom{0}\\\phantom{0}\\\phantom{0}
\end{tabular}\endgroup%
}}\!\right]$}%
{$\left[\!\llap{\phantom{%
\begingroup \smaller\smaller\smaller\begin{tabular}{@{}c@{}}%
0\\0\\0
\end{tabular}\endgroup%
}}\right.$}%
\begingroup \smaller\smaller\smaller\begin{tabular}{@{}c@{}}%
12\\2\\0
\end{tabular}\endgroup%
\kern3pt%
\begingroup \smaller\smaller\smaller\begin{tabular}{@{}c@{}}%
30\\3\\-5
\end{tabular}\endgroup%
\kern3pt%
\begingroup \smaller\smaller\smaller\begin{tabular}{@{}c@{}}%
16\\0\\-4
\end{tabular}\endgroup%
{$\left.\llap{\phantom{%
\begingroup \smaller\smaller\smaller\begin{tabular}{@{}c@{}}%
0\\0\\0
\end{tabular}\endgroup%
}}\!\right]$}%
}%
\ifdim\wd\matricesbox>\halfwidth\myboxwidth=\hsize\else\myboxwidth=\halfwidth\fi
\vbox{%
\ifdim\myboxwidth=\hsize
\setbox\onelinebox=\hbox{%
\vbox{\hbox{%
$\Pi_{8,18}$ spans $L_{19.5}$%
}\hbox{%
$|22|22|22|22\rtimes D_{4}$%
}%
}%
\hfill\copy\matricesbox
}%
\ifdim\wd\onelinebox>\myboxwidth
\hbox to \myboxwidth{%
$\Pi_{8,18}$ spans $L_{19.5}$%
\hfil
$|22|22|22|22\rtimes D_{4}$%
}%
\box\matricesbox
\else
\hbox to \myboxwidth{%
\unhbox\onelinebox
}%
\fi
\else
\hbox to \myboxwidth{%
$\Pi_{8,18}$ spans $L_{19.5}$%
\hfil}%
\hbox to \myboxwidth{%
$|22|22|22|22\rtimes D_{4}$%
\hfil}%
\box\matricesbox
\fi
}%
\hfill\discretionary{}{}{}%
\setbox\matricesbox=\hbox{%
{$\left[\!\llap{\phantom{%
\begingroup \smaller\smaller\smaller\begin{tabular}{@{}c@{}}%
\phantom{0}\\\phantom{0}\\\phantom{0}
\end{tabular}\endgroup%
}}\right.$}%
\begingroup \smaller\smaller\smaller\begin{tabular}{@{}c@{}}%
-7/8\\\phantom{0}\\\phantom{0}
\end{tabular}\endgroup%
\kern3pt%
\begingroup \smaller\smaller\smaller\begin{tabular}{@{}c@{}}%
\phantom{0}\\1\\\phantom{0}
\end{tabular}\endgroup%
\kern3pt%
\begingroup \smaller\smaller\smaller\begin{tabular}{@{}c@{}}%
\phantom{0}\\\phantom{0}\\1/2
\end{tabular}\endgroup%
{$\left.\llap{\phantom{%
\begingroup \smaller\smaller\smaller\begin{tabular}{@{}c@{}}%
\phantom{0}\\\phantom{0}\\\phantom{0}
\end{tabular}\endgroup%
}}\!\right]$}%
{$\left[\!\llap{\phantom{%
\begingroup \smaller\smaller\smaller\begin{tabular}{@{}c@{}}%
0\\0\\0
\end{tabular}\endgroup%
}}\right.$}%
\begingroup \smaller\smaller\smaller\begin{tabular}{@{}c@{}}%
4\\4\\2
\end{tabular}\endgroup%
\kern3pt%
\begingroup \smaller\smaller\smaller\begin{tabular}{@{}c@{}}%
2\\1\\3
\end{tabular}\endgroup%
{$\left.\llap{\phantom{%
\begingroup \smaller\smaller\smaller\begin{tabular}{@{}c@{}}%
0\\0\\0
\end{tabular}\endgroup%
}}\!\right]$}%
}%
\ifdim\wd\matricesbox>\halfwidth\myboxwidth=\hsize\else\myboxwidth=\halfwidth\fi
\vbox{%
\ifdim\myboxwidth=\hsize
\setbox\onelinebox=\hbox{%
\vbox{\hbox{%
$\Pi_{8,19}$ spans $L_{8.1}$%
}\hbox{%
$2\slashtwo2\slashtwo2\slashtwo2\slashtwo\rtimes D_{4}$%
}%
}%
\hfill\copy\matricesbox
}%
\ifdim\wd\onelinebox>\myboxwidth
\hbox to \myboxwidth{%
$\Pi_{8,19}$ spans $L_{8.1}$%
\hfil
$2\slashtwo2\slashtwo2\slashtwo2\slashtwo\rtimes D_{4}$%
}%
\box\matricesbox
\else
\hbox to \myboxwidth{%
\unhbox\onelinebox
}%
\fi
\else
\hbox to \myboxwidth{%
$\Pi_{8,19}$ spans $L_{8.1}$%
\hfil}%
\hbox to \myboxwidth{%
$2\slashtwo2\slashtwo2\slashtwo2\slashtwo\rtimes D_{4}$%
\hfil}%
\box\matricesbox
\fi
}%
\hfill\discretionary{}{}{}%
\setbox\matricesbox=\hbox{%
{$\left[\!\llap{\phantom{%
\begingroup \smaller\smaller\smaller\begin{tabular}{@{}c@{}}%
\phantom{0}\\\phantom{0}\\\phantom{0}
\end{tabular}\endgroup%
}}\right.$}%
\begingroup \smaller\smaller\smaller\begin{tabular}{@{}c@{}}%
-1/2\\\phantom{0}\\\phantom{0}
\end{tabular}\endgroup%
\kern3pt%
\begingroup \smaller\smaller\smaller\begin{tabular}{@{}c@{}}%
\phantom{0}\\15/2\\\phantom{0}
\end{tabular}\endgroup%
\kern3pt%
\begingroup \smaller\smaller\smaller\begin{tabular}{@{}c@{}}%
\phantom{0}\\\phantom{0}\\10
\end{tabular}\endgroup%
{$\left.\llap{\phantom{%
\begingroup \smaller\smaller\smaller\begin{tabular}{@{}c@{}}%
\phantom{0}\\\phantom{0}\\\phantom{0}
\end{tabular}\endgroup%
}}\!\right]$}%
{$\left[\!\llap{\phantom{%
\begingroup \smaller\smaller\smaller\begin{tabular}{@{}c@{}}%
0\\0\\0
\end{tabular}\endgroup%
}}\right.$}%
\begingroup \smaller\smaller\smaller\begin{tabular}{@{}c@{}}%
3\\-1\\0
\end{tabular}\endgroup%
\kern3pt%
\begingroup \smaller\smaller\smaller\begin{tabular}{@{}c@{}}%
5\\-1\\-1
\end{tabular}\endgroup%
\kern3pt%
\begingroup \smaller\smaller\smaller\begin{tabular}{@{}c@{}}%
8\\0\\-2
\end{tabular}\endgroup%
{$\left.\llap{\phantom{%
\begingroup \smaller\smaller\smaller\begin{tabular}{@{}c@{}}%
0\\0\\0
\end{tabular}\endgroup%
}}\!\right]$}%
}%
\ifdim\wd\matricesbox>\halfwidth\myboxwidth=\hsize\else\myboxwidth=\halfwidth\fi
\vbox{%
\ifdim\myboxwidth=\hsize
\setbox\onelinebox=\hbox{%
\vbox{\hbox{%
$\Pi_{8,20}$ spans $L_{17.11}$%
}\hbox{%
$|22|22|22|22\rtimes D_{4}$%
}%
}%
\hfill\copy\matricesbox
}%
\ifdim\wd\onelinebox>\myboxwidth
\hbox to \myboxwidth{%
$\Pi_{8,20}$ spans $L_{17.11}$%
\hfil
$|22|22|22|22\rtimes D_{4}$%
}%
\box\matricesbox
\else
\hbox to \myboxwidth{%
\unhbox\onelinebox
}%
\fi
\else
\hbox to \myboxwidth{%
$\Pi_{8,20}$ spans $L_{17.11}$%
\hfil}%
\hbox to \myboxwidth{%
$|22|22|22|22\rtimes D_{4}$%
\hfil}%
\box\matricesbox
\fi
}%
\hfill\discretionary{}{}{}%
\setbox\matricesbox=\hbox{%
{$\left[\!\llap{\phantom{%
\begingroup \smaller\smaller\smaller\begin{tabular}{@{}c@{}}%
\phantom{0}\\\phantom{0}\\\phantom{0}
\end{tabular}\endgroup%
}}\right.$}%
\begingroup \smaller\smaller\smaller\begin{tabular}{@{}c@{}}%
-1/8\\\phantom{0}\\\phantom{0}
\end{tabular}\endgroup%
\kern3pt%
\begingroup \smaller\smaller\smaller\begin{tabular}{@{}c@{}}%
\phantom{0}\\63/2\\\phantom{0}
\end{tabular}\endgroup%
\kern3pt%
\begingroup \smaller\smaller\smaller\begin{tabular}{@{}c@{}}%
\phantom{0}\\\phantom{0}\\3
\end{tabular}\endgroup%
{$\left.\llap{\phantom{%
\begingroup \smaller\smaller\smaller\begin{tabular}{@{}c@{}}%
\phantom{0}\\\phantom{0}\\\phantom{0}
\end{tabular}\endgroup%
}}\!\right]$}%
{$\left[\!\llap{\phantom{%
\begingroup \smaller\smaller\smaller\begin{tabular}{@{}c@{}}%
0\\0\\0
\end{tabular}\endgroup%
}}\right.$}%
\begingroup \smaller\smaller\smaller\begin{tabular}{@{}c@{}}%
28\\2\\0
\end{tabular}\endgroup%
\kern3pt%
\begingroup \smaller\smaller\smaller\begin{tabular}{@{}c@{}}%
18\\1\\3
\end{tabular}\endgroup%
\kern3pt%
\begingroup \smaller\smaller\smaller\begin{tabular}{@{}c@{}}%
16\\0\\4
\end{tabular}\endgroup%
{$\left.\llap{\phantom{%
\begingroup \smaller\smaller\smaller\begin{tabular}{@{}c@{}}%
0\\0\\0
\end{tabular}\endgroup%
}}\!\right]$}%
}%
\ifdim\wd\matricesbox>\halfwidth\myboxwidth=\hsize\else\myboxwidth=\halfwidth\fi
\vbox{%
\ifdim\myboxwidth=\hsize
\setbox\onelinebox=\hbox{%
\vbox{\hbox{%
$\Pi_{8,21}$ spans $L_{25.12}$%
}\hbox{%
$2|22|22|22|2\rtimes D_{4}$%
}%
}%
\hfill\copy\matricesbox
}%
\ifdim\wd\onelinebox>\myboxwidth
\hbox to \myboxwidth{%
$\Pi_{8,21}$ spans $L_{25.12}$%
\hfil
$2|22|22|22|2\rtimes D_{4}$%
}%
\box\matricesbox
\else
\hbox to \myboxwidth{%
\unhbox\onelinebox
}%
\fi
\else
\hbox to \myboxwidth{%
$\Pi_{8,21}$ spans $L_{25.12}$%
\hfil}%
\hbox to \myboxwidth{%
$2|22|22|22|2\rtimes D_{4}$%
\hfil}%
\box\matricesbox
\fi
}%
\hfill\discretionary{}{}{}%
\setbox\matricesbox=\hbox{%
{$\left[\!\llap{\phantom{%
\begingroup \smaller\smaller\smaller\begin{tabular}{@{}c@{}}%
\phantom{0}\\\phantom{0}\\\phantom{0}
\end{tabular}\endgroup%
}}\right.$}%
\begingroup \smaller\smaller\smaller\begin{tabular}{@{}c@{}}%
-1\\\phantom{0}\\\phantom{0}
\end{tabular}\endgroup%
\kern3pt%
\begingroup \smaller\smaller\smaller\begin{tabular}{@{}c@{}}%
\phantom{0}\\14\\\phantom{0}
\end{tabular}\endgroup%
\kern3pt%
\begingroup \smaller\smaller\smaller\begin{tabular}{@{}c@{}}%
\phantom{0}\\\phantom{0}\\6
\end{tabular}\endgroup%
{$\left.\llap{\phantom{%
\begingroup \smaller\smaller\smaller\begin{tabular}{@{}c@{}}%
\phantom{0}\\\phantom{0}\\\phantom{0}
\end{tabular}\endgroup%
}}\!\right]$}%
{$\left[\!\llap{\phantom{%
\begingroup \smaller\smaller\smaller\begin{tabular}{@{}c@{}}%
0\\0\\0
\end{tabular}\endgroup%
}}\right.$}%
\begingroup \smaller\smaller\smaller\begin{tabular}{@{}c@{}}%
7\\-2\\0
\end{tabular}\endgroup%
\kern3pt%
\begingroup \smaller\smaller\smaller\begin{tabular}{@{}c@{}}%
4\\-1\\-1
\end{tabular}\endgroup%
\kern3pt%
\begingroup \smaller\smaller\smaller\begin{tabular}{@{}c@{}}%
2\\0\\-1
\end{tabular}\endgroup%
{$\left.\llap{\phantom{%
\begingroup \smaller\smaller\smaller\begin{tabular}{@{}c@{}}%
0\\0\\0
\end{tabular}\endgroup%
}}\!\right]$}%
}%
\ifdim\wd\matricesbox>\halfwidth\myboxwidth=\hsize\else\myboxwidth=\halfwidth\fi
\vbox{%
\ifdim\myboxwidth=\hsize
\setbox\onelinebox=\hbox{%
\vbox{\hbox{%
$\Pi_{8,22}$ spans $L_{24.8}$%
}\hbox{%
$2|24|42|24|4\rtimes D_{4}$%
}%
}%
\hfill\copy\matricesbox
}%
\ifdim\wd\onelinebox>\myboxwidth
\hbox to \myboxwidth{%
$\Pi_{8,22}$ spans $L_{24.8}$%
\hfil
$2|24|42|24|4\rtimes D_{4}$%
}%
\box\matricesbox
\else
\hbox to \myboxwidth{%
\unhbox\onelinebox
}%
\fi
\else
\hbox to \myboxwidth{%
$\Pi_{8,22}$ spans $L_{24.8}$%
\hfil}%
\hbox to \myboxwidth{%
$2|24|42|24|4\rtimes D_{4}$%
\hfil}%
\box\matricesbox
\fi
}%
\hfill\discretionary{}{}{}%
\setbox\matricesbox=\hbox{%
{$\left[\!\llap{\phantom{%
\begingroup \smaller\smaller\smaller\begin{tabular}{@{}c@{}}%
\phantom{0}\\\phantom{0}\\\phantom{0}
\end{tabular}\endgroup%
}}\right.$}%
\begingroup \smaller\smaller\smaller\begin{tabular}{@{}c@{}}%
-1/4\\\phantom{0}\\\phantom{0}
\end{tabular}\endgroup%
\kern3pt%
\begingroup \smaller\smaller\smaller\begin{tabular}{@{}c@{}}%
\phantom{0}\\45/2\\\phantom{0}
\end{tabular}\endgroup%
\kern3pt%
\begingroup \smaller\smaller\smaller\begin{tabular}{@{}c@{}}%
\phantom{0}\\\phantom{0}\\1/2
\end{tabular}\endgroup%
{$\left.\llap{\phantom{%
\begingroup \smaller\smaller\smaller\begin{tabular}{@{}c@{}}%
\phantom{0}\\\phantom{0}\\\phantom{0}
\end{tabular}\endgroup%
}}\!\right]$}%
{$\left[\!\llap{\phantom{%
\begingroup \smaller\smaller\smaller\begin{tabular}{@{}c@{}}%
0\\0\\0
\end{tabular}\endgroup%
}}\right.$}%
\begingroup \smaller\smaller\smaller\begin{tabular}{@{}c@{}}%
36\\-4\\0
\end{tabular}\endgroup%
\kern3pt%
\begingroup \smaller\smaller\smaller\begin{tabular}{@{}c@{}}%
10\\-1\\5
\end{tabular}\endgroup%
\kern3pt%
\begingroup \smaller\smaller\smaller\begin{tabular}{@{}c@{}}%
4\\0\\4
\end{tabular}\endgroup%
{$\left.\llap{\phantom{%
\begingroup \smaller\smaller\smaller\begin{tabular}{@{}c@{}}%
0\\0\\0
\end{tabular}\endgroup%
}}\!\right]$}%
}%
\ifdim\wd\matricesbox>\halfwidth\myboxwidth=\hsize\else\myboxwidth=\halfwidth\fi
\vbox{%
\ifdim\myboxwidth=\hsize
\setbox\onelinebox=\hbox{%
\vbox{\hbox{%
$\Pi_{8,23}$ spans $L_{165.1}$%
}\hbox{%
$2|22|22|22|2\rtimes D_{4}$%
}%
}%
\hfill\copy\matricesbox
}%
\ifdim\wd\onelinebox>\myboxwidth
\hbox to \myboxwidth{%
$\Pi_{8,23}$ spans $L_{165.1}$%
\hfil
$2|22|22|22|2\rtimes D_{4}$%
}%
\box\matricesbox
\else
\hbox to \myboxwidth{%
\unhbox\onelinebox
}%
\fi
\else
\hbox to \myboxwidth{%
$\Pi_{8,23}$ spans $L_{165.1}$%
\hfil}%
\hbox to \myboxwidth{%
$2|22|22|22|2\rtimes D_{4}$%
\hfil}%
\box\matricesbox
\fi
}%
\hfill\discretionary{}{}{}%
\setbox\matricesbox=\hbox{%
{$\left[\!\llap{\phantom{%
\begingroup \smaller\smaller\smaller\begin{tabular}{@{}c@{}}%
\phantom{0}\\\phantom{0}\\\phantom{0}
\end{tabular}\endgroup%
}}\right.$}%
\begingroup \smaller\smaller\smaller\begin{tabular}{@{}c@{}}%
-1/4\\\phantom{0}\\\phantom{0}
\end{tabular}\endgroup%
\kern3pt%
\begingroup \smaller\smaller\smaller\begin{tabular}{@{}c@{}}%
\phantom{0}\\7\\\phantom{0}
\end{tabular}\endgroup%
\kern3pt%
\begingroup \smaller\smaller\smaller\begin{tabular}{@{}c@{}}%
\phantom{0}\\\phantom{0}\\35
\end{tabular}\endgroup%
{$\left.\llap{\phantom{%
\begingroup \smaller\smaller\smaller\begin{tabular}{@{}c@{}}%
\phantom{0}\\\phantom{0}\\\phantom{0}
\end{tabular}\endgroup%
}}\!\right]$}%
{$\left[\!\llap{\phantom{%
\begingroup \smaller\smaller\smaller\begin{tabular}{@{}c@{}}%
0\\0\\0
\end{tabular}\endgroup%
}}\right.$}%
\begingroup \smaller\smaller\smaller\begin{tabular}{@{}c@{}}%
14\\3\\0
\end{tabular}\endgroup%
\kern3pt%
\begingroup \smaller\smaller\smaller\begin{tabular}{@{}c@{}}%
14\\2\\1
\end{tabular}\endgroup%
\kern3pt%
\begingroup \smaller\smaller\smaller\begin{tabular}{@{}c@{}}%
10\\0\\1
\end{tabular}\endgroup%
{$\left.\llap{\phantom{%
\begingroup \smaller\smaller\smaller\begin{tabular}{@{}c@{}}%
0\\0\\0
\end{tabular}\endgroup%
}}\!\right]$}%
}%
\ifdim\wd\matricesbox>\halfwidth\myboxwidth=\hsize\else\myboxwidth=\halfwidth\fi
\vbox{%
\ifdim\myboxwidth=\hsize
\setbox\onelinebox=\hbox{%
\vbox{\hbox{%
$\Pi_{8,24}$ spans $L_{37.7}$%
}\hbox{%
$3|32|23|32|2\rtimes D_{4}$%
}%
}%
\hfill\copy\matricesbox
}%
\ifdim\wd\onelinebox>\myboxwidth
\hbox to \myboxwidth{%
$\Pi_{8,24}$ spans $L_{37.7}$%
\hfil
$3|32|23|32|2\rtimes D_{4}$%
}%
\box\matricesbox
\else
\hbox to \myboxwidth{%
\unhbox\onelinebox
}%
\fi
\else
\hbox to \myboxwidth{%
$\Pi_{8,24}$ spans $L_{37.7}$%
\hfil}%
\hbox to \myboxwidth{%
$3|32|23|32|2\rtimes D_{4}$%
\hfil}%
\box\matricesbox
\fi
}%
\hfill\discretionary{}{}{}%
\setbox\matricesbox=\hbox{%
{$\left[\!\llap{\phantom{%
\begingroup \smaller\smaller\smaller\begin{tabular}{@{}c@{}}%
\phantom{0}\\\phantom{0}\\\phantom{0}
\end{tabular}\endgroup%
}}\right.$}%
\begingroup \smaller\smaller\smaller\begin{tabular}{@{}c@{}}%
-3\\\phantom{0}\\\phantom{0}
\end{tabular}\endgroup%
\kern3pt%
\begingroup \smaller\smaller\smaller\begin{tabular}{@{}c@{}}%
\phantom{0}\\1\\\phantom{0}
\end{tabular}\endgroup%
\kern3pt%
\begingroup \smaller\smaller\smaller\begin{tabular}{@{}c@{}}%
\phantom{0}\\\phantom{0}\\3
\end{tabular}\endgroup%
{$\left.\llap{\phantom{%
\begingroup \smaller\smaller\smaller\begin{tabular}{@{}c@{}}%
\phantom{0}\\\phantom{0}\\\phantom{0}
\end{tabular}\endgroup%
}}\!\right]$}%
{$\left[\!\llap{\phantom{%
\begingroup \smaller\smaller\smaller\begin{tabular}{@{}c@{}}%
0\\0\\0
\end{tabular}\endgroup%
}}\right.$}%
\begingroup \smaller\smaller\smaller\begin{tabular}{@{}c@{}}%
4\\7\\1
\end{tabular}\endgroup%
\kern3pt%
\begingroup \smaller\smaller\smaller\begin{tabular}{@{}c@{}}%
1\\1\\1
\end{tabular}\endgroup%
{$\left.\llap{\phantom{%
\begingroup \smaller\smaller\smaller\begin{tabular}{@{}c@{}}%
0\\0\\0
\end{tabular}\endgroup%
}}\!\right]$}%
}%
\ifdim\wd\matricesbox>\halfwidth\myboxwidth=\hsize\else\myboxwidth=\halfwidth\fi
\vbox{%
\ifdim\myboxwidth=\hsize
\setbox\onelinebox=\hbox{%
\vbox{\hbox{%
$\Pi_{8,25}=\hbox{GN}_{50}$ spans $L_{7.7}$%
}\hbox{%
$\slashthree\infty\slashinfty\infty\slashthree\infty\slashinfty\infty\rtimes D_{4}$%
}%
}%
\hfill\copy\matricesbox
}%
\ifdim\wd\onelinebox>\myboxwidth
\hbox to \myboxwidth{%
$\Pi_{8,25}=\hbox{GN}_{50}$ spans $L_{7.7}$%
\hfil
$\slashthree\infty\slashinfty\infty\slashthree\infty\slashinfty\infty\rtimes D_{4}$%
}%
\box\matricesbox
\else
\hbox to \myboxwidth{%
\unhbox\onelinebox
}%
\fi
\else
\hbox to \myboxwidth{%
$\Pi_{8,25}=\hbox{GN}_{50}$ spans $L_{7.7}$%
\hfil}%
\hbox to \myboxwidth{%
$\slashthree\infty\slashinfty\infty\slashthree\infty\slashinfty\infty\rtimes D_{4}$%
\hfil}%
\box\matricesbox
\fi
}%
\hfill\discretionary{}{}{}%
\setbox\matricesbox=\hbox{%
{$\left[\!\llap{\phantom{%
\begingroup \smaller\smaller\smaller\begin{tabular}{@{}c@{}}%
\phantom{0}\\\phantom{0}\\\phantom{0}
\end{tabular}\endgroup%
}}\right.$}%
\begingroup \smaller\smaller\smaller\begin{tabular}{@{}c@{}}%
-1\\\phantom{0}\\\phantom{0}
\end{tabular}\endgroup%
\kern3pt%
\begingroup \smaller\smaller\smaller\begin{tabular}{@{}c@{}}%
\phantom{0}\\7/2\\\phantom{0}
\end{tabular}\endgroup%
\kern3pt%
\begingroup \smaller\smaller\smaller\begin{tabular}{@{}c@{}}%
\phantom{0}\\\phantom{0}\\1/2
\end{tabular}\endgroup%
{$\left.\llap{\phantom{%
\begingroup \smaller\smaller\smaller\begin{tabular}{@{}c@{}}%
\phantom{0}\\\phantom{0}\\\phantom{0}
\end{tabular}\endgroup%
}}\!\right]$}%
{$\left[\!\llap{\phantom{%
\begingroup \smaller\smaller\smaller\begin{tabular}{@{}c@{}}%
0\\0\\0
\end{tabular}\endgroup%
}}\right.$}%
\begingroup \smaller\smaller\smaller\begin{tabular}{@{}c@{}}%
7\\4\\0
\end{tabular}\endgroup%
\kern3pt%
\begingroup \smaller\smaller\smaller\begin{tabular}{@{}c@{}}%
7\\3\\-7
\end{tabular}\endgroup%
\kern3pt%
\begingroup \smaller\smaller\smaller\begin{tabular}{@{}c@{}}%
1\\0\\-2
\end{tabular}\endgroup%
{$\left.\llap{\phantom{%
\begingroup \smaller\smaller\smaller\begin{tabular}{@{}c@{}}%
0\\0\\0
\end{tabular}\endgroup%
}}\!\right]$}%
}%
\ifdim\wd\matricesbox>\halfwidth\myboxwidth=\hsize\else\myboxwidth=\halfwidth\fi
\vbox{%
\ifdim\myboxwidth=\hsize
\setbox\onelinebox=\hbox{%
\vbox{\hbox{%
$\Pi_{8,26}$ spans $L_{9.4}$%
}\hbox{%
$\infty|\infty2|2\infty|\infty2|2\rtimes D_{4}$%
}%
}%
\hfill\copy\matricesbox
}%
\ifdim\wd\onelinebox>\myboxwidth
\hbox to \myboxwidth{%
$\Pi_{8,26}$ spans $L_{9.4}$%
\hfil
$\infty|\infty2|2\infty|\infty2|2\rtimes D_{4}$%
}%
\box\matricesbox
\else
\hbox to \myboxwidth{%
\unhbox\onelinebox
}%
\fi
\else
\hbox to \myboxwidth{%
$\Pi_{8,26}$ spans $L_{9.4}$%
\hfil}%
\hbox to \myboxwidth{%
$\infty|\infty2|2\infty|\infty2|2\rtimes D_{4}$%
\hfil}%
\box\matricesbox
\fi
}%
\hfill\discretionary{}{}{}%
\setbox\matricesbox=\hbox{%
{$\left[\!\llap{\phantom{%
\begingroup \smaller\smaller\smaller\begin{tabular}{@{}c@{}}%
\phantom{0}\\\phantom{0}\\\phantom{0}
\end{tabular}\endgroup%
}}\right.$}%
\begingroup \smaller\smaller\smaller\begin{tabular}{@{}c@{}}%
-1\\\phantom{0}\\\phantom{0}
\end{tabular}\endgroup%
\kern3pt%
\begingroup \smaller\smaller\smaller\begin{tabular}{@{}c@{}}%
\phantom{0}\\3/2\\\phantom{0}
\end{tabular}\endgroup%
\kern3pt%
\begingroup \smaller\smaller\smaller\begin{tabular}{@{}c@{}}%
\phantom{0}\\\phantom{0}\\9/2
\end{tabular}\endgroup%
{$\left.\llap{\phantom{%
\begingroup \smaller\smaller\smaller\begin{tabular}{@{}c@{}}%
\phantom{0}\\\phantom{0}\\\phantom{0}
\end{tabular}\endgroup%
}}\!\right]$}%
{$\left[\!\llap{\phantom{%
\begingroup \smaller\smaller\smaller\begin{tabular}{@{}c@{}}%
0\\0\\0
\end{tabular}\endgroup%
}}\right.$}%
\begingroup \smaller\smaller\smaller\begin{tabular}{@{}c@{}}%
6\\-5\\-1
\end{tabular}\endgroup%
\kern3pt%
\begingroup \smaller\smaller\smaller\begin{tabular}{@{}c@{}}%
2\\-1\\-1
\end{tabular}\endgroup%
{$\left.\llap{\phantom{%
\begingroup \smaller\smaller\smaller\begin{tabular}{@{}c@{}}%
0\\0\\0
\end{tabular}\endgroup%
}}\!\right]$}%
}%
\ifdim\wd\matricesbox>\halfwidth\myboxwidth=\hsize\else\myboxwidth=\halfwidth\fi
\vbox{%
\ifdim\myboxwidth=\hsize
\setbox\onelinebox=\hbox{%
\vbox{\hbox{%
$\Pi_{8,27}$ spans $L_{168.1}$%
}\hbox{%
$\slashthree2\slashthree2\slashthree2\slashthree2\rtimes D_{4}$%
}%
}%
\hfill\copy\matricesbox
}%
\ifdim\wd\onelinebox>\myboxwidth
\hbox to \myboxwidth{%
$\Pi_{8,27}$ spans $L_{168.1}$%
\hfil
$\slashthree2\slashthree2\slashthree2\slashthree2\rtimes D_{4}$%
}%
\box\matricesbox
\else
\hbox to \myboxwidth{%
\unhbox\onelinebox
}%
\fi
\else
\hbox to \myboxwidth{%
$\Pi_{8,27}$ spans $L_{168.1}$%
\hfil}%
\hbox to \myboxwidth{%
$\slashthree2\slashthree2\slashthree2\slashthree2\rtimes D_{4}$%
\hfil}%
\box\matricesbox
\fi
}%
\hfill\discretionary{}{}{}%
\setbox\matricesbox=\hbox{%
{$\left[\!\llap{\phantom{%
\begingroup \smaller\smaller\smaller\begin{tabular}{@{}c@{}}%
\phantom{0}\\\phantom{0}\\\phantom{0}
\end{tabular}\endgroup%
}}\right.$}%
\begingroup \smaller\smaller\smaller\begin{tabular}{@{}c@{}}%
-1\\\phantom{0}\\\phantom{0}
\end{tabular}\endgroup%
\kern3pt%
\begingroup \smaller\smaller\smaller\begin{tabular}{@{}c@{}}%
\phantom{0}\\6\\\phantom{0}
\end{tabular}\endgroup%
\kern3pt%
\begingroup \smaller\smaller\smaller\begin{tabular}{@{}c@{}}%
\phantom{0}\\\phantom{0}\\12
\end{tabular}\endgroup%
{$\left.\llap{\phantom{%
\begingroup \smaller\smaller\smaller\begin{tabular}{@{}c@{}}%
\phantom{0}\\\phantom{0}\\\phantom{0}
\end{tabular}\endgroup%
}}\!\right]$}%
{$\left[\!\llap{\phantom{%
\begingroup \smaller\smaller\smaller\begin{tabular}{@{}c@{}}%
0\\0\\0
\end{tabular}\endgroup%
}}\right.$}%
\begingroup \smaller\smaller\smaller\begin{tabular}{@{}c@{}}%
2\\-1\\0
\end{tabular}\endgroup%
\kern3pt%
\begingroup \smaller\smaller\smaller\begin{tabular}{@{}c@{}}%
8\\-2\\-2
\end{tabular}\endgroup%
\kern3pt%
\begingroup \smaller\smaller\smaller\begin{tabular}{@{}c@{}}%
3\\0\\-1
\end{tabular}\endgroup%
{$\left.\llap{\phantom{%
\begingroup \smaller\smaller\smaller\begin{tabular}{@{}c@{}}%
0\\0\\0
\end{tabular}\endgroup%
}}\!\right]$}%
}%
\ifdim\wd\matricesbox>\halfwidth\myboxwidth=\hsize\else\myboxwidth=\halfwidth\fi
\vbox{%
\ifdim\myboxwidth=\hsize
\setbox\onelinebox=\hbox{%
\vbox{\hbox{%
$\Pi_{8,28}$ spans $L_{150.29}$%
}\hbox{%
$|\infty2|2\infty|\infty2|2\infty\rtimes D_{4}$%
}%
}%
\hfill\copy\matricesbox
}%
\ifdim\wd\onelinebox>\myboxwidth
\hbox to \myboxwidth{%
$\Pi_{8,28}$ spans $L_{150.29}$%
\hfil
$|\infty2|2\infty|\infty2|2\infty\rtimes D_{4}$%
}%
\box\matricesbox
\else
\hbox to \myboxwidth{%
\unhbox\onelinebox
}%
\fi
\else
\hbox to \myboxwidth{%
$\Pi_{8,28}$ spans $L_{150.29}$%
\hfil}%
\hbox to \myboxwidth{%
$|\infty2|2\infty|\infty2|2\infty\rtimes D_{4}$%
\hfil}%
\box\matricesbox
\fi
}%
\hfill\discretionary{}{}{}%
\setbox\matricesbox=\hbox{%
{$\left[\!\llap{\phantom{%
\begingroup \smaller\smaller\smaller\begin{tabular}{@{}c@{}}%
\phantom{0}\\\phantom{0}\\\phantom{0}
\end{tabular}\endgroup%
}}\right.$}%
\begingroup \smaller\smaller\smaller\begin{tabular}{@{}c@{}}%
-1/2\\\phantom{0}\\\phantom{0}
\end{tabular}\endgroup%
\kern3pt%
\begingroup \smaller\smaller\smaller\begin{tabular}{@{}c@{}}%
\phantom{0}\\7/2\\\phantom{0}
\end{tabular}\endgroup%
\kern3pt%
\begingroup \smaller\smaller\smaller\begin{tabular}{@{}c@{}}%
\phantom{0}\\\phantom{0}\\4
\end{tabular}\endgroup%
{$\left.\llap{\phantom{%
\begingroup \smaller\smaller\smaller\begin{tabular}{@{}c@{}}%
\phantom{0}\\\phantom{0}\\\phantom{0}
\end{tabular}\endgroup%
}}\!\right]$}%
{$\left[\!\llap{\phantom{%
\begingroup \smaller\smaller\smaller\begin{tabular}{@{}c@{}}%
0\\0\\0
\end{tabular}\endgroup%
}}\right.$}%
\begingroup \smaller\smaller\smaller\begin{tabular}{@{}c@{}}%
7\\-3\\0
\end{tabular}\endgroup%
\kern3pt%
\begingroup \smaller\smaller\smaller\begin{tabular}{@{}c@{}}%
28\\-8\\-7
\end{tabular}\endgroup%
\kern3pt%
\begingroup \smaller\smaller\smaller\begin{tabular}{@{}c@{}}%
2\\0\\-1
\end{tabular}\endgroup%
{$\left.\llap{\phantom{%
\begingroup \smaller\smaller\smaller\begin{tabular}{@{}c@{}}%
0\\0\\0
\end{tabular}\endgroup%
}}\!\right]$}%
}%
\ifdim\wd\matricesbox>\halfwidth\myboxwidth=\hsize\else\myboxwidth=\halfwidth\fi
\vbox{%
\ifdim\myboxwidth=\hsize
\setbox\onelinebox=\hbox{%
\vbox{\hbox{%
$\Pi_{8,29}$ spans $L_{9.6}$%
}\hbox{%
$|\infty2|2\infty|\infty2|2\infty\rtimes D_{4}$%
}%
}%
\hfill\copy\matricesbox
}%
\ifdim\wd\onelinebox>\myboxwidth
\hbox to \myboxwidth{%
$\Pi_{8,29}$ spans $L_{9.6}$%
\hfil
$|\infty2|2\infty|\infty2|2\infty\rtimes D_{4}$%
}%
\box\matricesbox
\else
\hbox to \myboxwidth{%
\unhbox\onelinebox
}%
\fi
\else
\hbox to \myboxwidth{%
$\Pi_{8,29}$ spans $L_{9.6}$%
\hfil}%
\hbox to \myboxwidth{%
$|\infty2|2\infty|\infty2|2\infty\rtimes D_{4}$%
\hfil}%
\box\matricesbox
\fi
}%
\hfill\discretionary{}{}{}%
\setbox\matricesbox=\hbox{%
{$\left[\!\llap{\phantom{%
\begingroup \smaller\smaller\smaller\begin{tabular}{@{}c@{}}%
\phantom{0}\\\phantom{0}\\\phantom{0}
\end{tabular}\endgroup%
}}\right.$}%
\begingroup \smaller\smaller\smaller\begin{tabular}{@{}c@{}}%
-1\\\phantom{0}\\\phantom{0}
\end{tabular}\endgroup%
\kern3pt%
\begingroup \smaller\smaller\smaller\begin{tabular}{@{}c@{}}%
\phantom{0}\\1\\\phantom{0}
\end{tabular}\endgroup%
\kern3pt%
\begingroup \smaller\smaller\smaller\begin{tabular}{@{}c@{}}%
\phantom{0}\\\phantom{0}\\5
\end{tabular}\endgroup%
{$\left.\llap{\phantom{%
\begingroup \smaller\smaller\smaller\begin{tabular}{@{}c@{}}%
\phantom{0}\\\phantom{0}\\\phantom{0}
\end{tabular}\endgroup%
}}\!\right]$}%
{$\left[\!\llap{\phantom{%
\begingroup \smaller\smaller\smaller\begin{tabular}{@{}c@{}}%
0\\0\\0
\end{tabular}\endgroup%
}}\right.$}%
\begingroup \smaller\smaller\smaller\begin{tabular}{@{}c@{}}%
5\\5\\1
\end{tabular}\endgroup%
\kern3pt%
\begingroup \smaller\smaller\smaller\begin{tabular}{@{}c@{}}%
2\\1\\1
\end{tabular}\endgroup%
{$\left.\llap{\phantom{%
\begingroup \smaller\smaller\smaller\begin{tabular}{@{}c@{}}%
0\\0\\0
\end{tabular}\endgroup%
}}\!\right]$}%
}%
\ifdim\wd\matricesbox>\halfwidth\myboxwidth=\hsize\else\myboxwidth=\halfwidth\fi
\vbox{%
\ifdim\myboxwidth=\hsize
\setbox\onelinebox=\hbox{%
\vbox{\hbox{%
$\Pi_{8,30}$ spans $L_{5.7}$%
}\hbox{%
$\slashinfty2\slashtwo2\slashinfty2\slashtwo2\rtimes D_{4}$%
}%
}%
\hfill\copy\matricesbox
}%
\ifdim\wd\onelinebox>\myboxwidth
\hbox to \myboxwidth{%
$\Pi_{8,30}$ spans $L_{5.7}$%
\hfil
$\slashinfty2\slashtwo2\slashinfty2\slashtwo2\rtimes D_{4}$%
}%
\box\matricesbox
\else
\hbox to \myboxwidth{%
\unhbox\onelinebox
}%
\fi
\else
\hbox to \myboxwidth{%
$\Pi_{8,30}$ spans $L_{5.7}$%
\hfil}%
\hbox to \myboxwidth{%
$\slashinfty2\slashtwo2\slashinfty2\slashtwo2\rtimes D_{4}$%
\hfil}%
\box\matricesbox
\fi
}%
\hfill\discretionary{}{}{}%
\setbox\matricesbox=\hbox{%
{$\left[\!\llap{\phantom{%
\begingroup \smaller\smaller\smaller\begin{tabular}{@{}c@{}}%
\phantom{0}\\\phantom{0}\\\phantom{0}
\end{tabular}\endgroup%
}}\right.$}%
\begingroup \smaller\smaller\smaller\begin{tabular}{@{}c@{}}%
-9/8\\\phantom{0}\\\phantom{0}
\end{tabular}\endgroup%
\kern3pt%
\begingroup \smaller\smaller\smaller\begin{tabular}{@{}c@{}}%
\phantom{0}\\2\\\phantom{0}
\end{tabular}\endgroup%
\kern3pt%
\begingroup \smaller\smaller\smaller\begin{tabular}{@{}c@{}}%
\phantom{0}\\\phantom{0}\\1/2
\end{tabular}\endgroup%
{$\left.\llap{\phantom{%
\begingroup \smaller\smaller\smaller\begin{tabular}{@{}c@{}}%
\phantom{0}\\\phantom{0}\\\phantom{0}
\end{tabular}\endgroup%
}}\!\right]$}%
{$\left[\!\llap{\phantom{%
\begingroup \smaller\smaller\smaller\begin{tabular}{@{}c@{}}%
0\\0\\0
\end{tabular}\endgroup%
}}\right.$}%
\begingroup \smaller\smaller\smaller\begin{tabular}{@{}c@{}}%
8\\6\\-4
\end{tabular}\endgroup%
\kern3pt%
\begingroup \smaller\smaller\smaller\begin{tabular}{@{}c@{}}%
2\\1\\-3
\end{tabular}\endgroup%
{$\left.\llap{\phantom{%
\begingroup \smaller\smaller\smaller\begin{tabular}{@{}c@{}}%
0\\0\\0
\end{tabular}\endgroup%
}}\!\right]$}%
}%
\ifdim\wd\matricesbox>\halfwidth\myboxwidth=\hsize\else\myboxwidth=\halfwidth\fi
\vbox{%
\ifdim\myboxwidth=\hsize
\setbox\onelinebox=\hbox{%
\vbox{\hbox{%
$\Pi_{8,31}=\hbox{GN}_{45}$ spans $L_{148.2}$%
}\hbox{%
$\slashinfty2\slashinfty2\slashinfty2\slashinfty2\rtimes D_{4}$%
}%
}%
\hfill\copy\matricesbox
}%
\ifdim\wd\onelinebox>\myboxwidth
\hbox to \myboxwidth{%
$\Pi_{8,31}=\hbox{GN}_{45}$ spans $L_{148.2}$%
\hfil
$\slashinfty2\slashinfty2\slashinfty2\slashinfty2\rtimes D_{4}$%
}%
\box\matricesbox
\else
\hbox to \myboxwidth{%
\unhbox\onelinebox
}%
\fi
\else
\hbox to \myboxwidth{%
$\Pi_{8,31}=\hbox{GN}_{45}$ spans $L_{148.2}$%
\hfil}%
\hbox to \myboxwidth{%
$\slashinfty2\slashinfty2\slashinfty2\slashinfty2\rtimes D_{4}$%
\hfil}%
\box\matricesbox
\fi
}%
\hfill\discretionary{}{}{}%
\setbox\matricesbox=\hbox{%
{$\left[\!\llap{\phantom{%
\begingroup \smaller\smaller\smaller\begin{tabular}{@{}c@{}}%
\phantom{0}\\\phantom{0}\\\phantom{0}
\end{tabular}\endgroup%
}}\right.$}%
\begingroup \smaller\smaller\smaller\begin{tabular}{@{}c@{}}%
-1/8\\\phantom{0}\\\phantom{0}
\end{tabular}\endgroup%
\kern3pt%
\begingroup \smaller\smaller\smaller\begin{tabular}{@{}c@{}}%
\phantom{0}\\21/2\\\phantom{0}
\end{tabular}\endgroup%
\kern3pt%
\begingroup \smaller\smaller\smaller\begin{tabular}{@{}c@{}}%
\phantom{0}\\\phantom{0}\\28
\end{tabular}\endgroup%
{$\left.\llap{\phantom{%
\begingroup \smaller\smaller\smaller\begin{tabular}{@{}c@{}}%
\phantom{0}\\\phantom{0}\\\phantom{0}
\end{tabular}\endgroup%
}}\!\right]$}%
{$\left[\!\llap{\phantom{%
\begingroup \smaller\smaller\smaller\begin{tabular}{@{}c@{}}%
0\\0\\0
\end{tabular}\endgroup%
}}\right.$}%
\begingroup \smaller\smaller\smaller\begin{tabular}{@{}c@{}}%
6\\-1\\0
\end{tabular}\endgroup%
\kern3pt%
\begingroup \smaller\smaller\smaller\begin{tabular}{@{}c@{}}%
112\\-8\\-6
\end{tabular}\endgroup%
\kern3pt%
\begingroup \smaller\smaller\smaller\begin{tabular}{@{}c@{}}%
42\\-1\\-3
\end{tabular}\endgroup%
\kern3pt%
\begingroup \smaller\smaller\smaller\begin{tabular}{@{}c@{}}%
14\\1\\-1
\end{tabular}\endgroup%
\kern3pt%
\begingroup \smaller\smaller\smaller\begin{tabular}{@{}c@{}}%
6\\1\\0
\end{tabular}\endgroup%
{$\left.\llap{\phantom{%
\begingroup \smaller\smaller\smaller\begin{tabular}{@{}c@{}}%
0\\0\\0
\end{tabular}\endgroup%
}}\!\right]$}%
}%
\ifdim\wd\matricesbox>\halfwidth\myboxwidth=\hsize\else\myboxwidth=\halfwidth\fi
\vbox{%
\ifdim\myboxwidth=\hsize
\setbox\onelinebox=\hbox{%
\vbox{\hbox{%
$\Pi_{8,32}$ spans $L_{22.8}$%
}\hbox{%
$|2222|2222\rtimes D_{2}$%
}%
}%
\hfill\copy\matricesbox
}%
\ifdim\wd\onelinebox>\myboxwidth
\hbox to \myboxwidth{%
$\Pi_{8,32}$ spans $L_{22.8}$%
\hfil
$|2222|2222\rtimes D_{2}$%
}%
\box\matricesbox
\else
\hbox to \myboxwidth{%
\unhbox\onelinebox
}%
\fi
\else
\hbox to \myboxwidth{%
$\Pi_{8,32}$ spans $L_{22.8}$%
\hfil}%
\hbox to \myboxwidth{%
$|2222|2222\rtimes D_{2}$%
\hfil}%
\box\matricesbox
\fi
}%
\hfill\discretionary{}{}{}%
\setbox\matricesbox=\hbox{%
{$\left[\!\llap{\phantom{%
\begingroup \smaller\smaller\smaller\begin{tabular}{@{}c@{}}%
\phantom{0}\\\phantom{0}\\\phantom{0}
\end{tabular}\endgroup%
}}\right.$}%
\begingroup \smaller\smaller\smaller\begin{tabular}{@{}c@{}}%
-1/2\\\phantom{0}\\\phantom{0}
\end{tabular}\endgroup%
\kern3pt%
\begingroup \smaller\smaller\smaller\begin{tabular}{@{}c@{}}%
\phantom{0}\\3/2\\\phantom{0}
\end{tabular}\endgroup%
\kern3pt%
\begingroup \smaller\smaller\smaller\begin{tabular}{@{}c@{}}%
\phantom{0}\\\phantom{0}\\5/2
\end{tabular}\endgroup%
{$\left.\llap{\phantom{%
\begingroup \smaller\smaller\smaller\begin{tabular}{@{}c@{}}%
\phantom{0}\\\phantom{0}\\\phantom{0}
\end{tabular}\endgroup%
}}\!\right]$}%
{$\left[\!\llap{\phantom{%
\begingroup \smaller\smaller\smaller\begin{tabular}{@{}c@{}}%
0\\0\\0
\end{tabular}\endgroup%
}}\right.$}%
\begingroup \smaller\smaller\smaller\begin{tabular}{@{}c@{}}%
6\\-4\\0
\end{tabular}\endgroup%
\kern3pt%
\begingroup \smaller\smaller\smaller\begin{tabular}{@{}c@{}}%
2\\-1\\1
\end{tabular}\endgroup%
\kern3pt%
\begingroup \smaller\smaller\smaller\begin{tabular}{@{}c@{}}%
6\\1\\3
\end{tabular}\endgroup%
\kern3pt%
\begingroup \smaller\smaller\smaller\begin{tabular}{@{}c@{}}%
10\\5\\3
\end{tabular}\endgroup%
\kern3pt%
\begingroup \smaller\smaller\smaller\begin{tabular}{@{}c@{}}%
6\\4\\0
\end{tabular}\endgroup%
{$\left.\llap{\phantom{%
\begingroup \smaller\smaller\smaller\begin{tabular}{@{}c@{}}%
0\\0\\0
\end{tabular}\endgroup%
}}\!\right]$}%
}%
\ifdim\wd\matricesbox>\halfwidth\myboxwidth=\hsize\else\myboxwidth=\halfwidth\fi
\vbox{%
\ifdim\myboxwidth=\hsize
\setbox\onelinebox=\hbox{%
\vbox{\hbox{%
$\Pi_{8,33}$ spans $L_{31.2}$%
}\hbox{%
$2|2222|222\rtimes D_{2}$%
}%
}%
\hfill\copy\matricesbox
}%
\ifdim\wd\onelinebox>\myboxwidth
\hbox to \myboxwidth{%
$\Pi_{8,33}$ spans $L_{31.2}$%
\hfil
$2|2222|222\rtimes D_{2}$%
}%
\box\matricesbox
\else
\hbox to \myboxwidth{%
\unhbox\onelinebox
}%
\fi
\else
\hbox to \myboxwidth{%
$\Pi_{8,33}$ spans $L_{31.2}$%
\hfil}%
\hbox to \myboxwidth{%
$2|2222|222\rtimes D_{2}$%
\hfil}%
\box\matricesbox
\fi
}%
\hfill\discretionary{}{}{}%
\setbox\matricesbox=\hbox{%
{$\left[\!\llap{\phantom{%
\begingroup \smaller\smaller\smaller\begin{tabular}{@{}c@{}}%
\phantom{0}\\\phantom{0}\\\phantom{0}
\end{tabular}\endgroup%
}}\right.$}%
\begingroup \smaller\smaller\smaller\begin{tabular}{@{}c@{}}%
-1/4\\\phantom{0}\\\phantom{0}
\end{tabular}\endgroup%
\kern3pt%
\begingroup \smaller\smaller\smaller\begin{tabular}{@{}c@{}}%
\phantom{0}\\15/4\\\phantom{0}
\end{tabular}\endgroup%
\kern3pt%
\begingroup \smaller\smaller\smaller\begin{tabular}{@{}c@{}}%
\phantom{0}\\\phantom{0}\\15/2
\end{tabular}\endgroup%
{$\left.\llap{\phantom{%
\begingroup \smaller\smaller\smaller\begin{tabular}{@{}c@{}}%
\phantom{0}\\\phantom{0}\\\phantom{0}
\end{tabular}\endgroup%
}}\!\right]$}%
{$\left[\!\llap{\phantom{%
\begingroup \smaller\smaller\smaller\begin{tabular}{@{}c@{}}%
0\\0\\0
\end{tabular}\endgroup%
}}\right.$}%
\begingroup \smaller\smaller\smaller\begin{tabular}{@{}c@{}}%
6\\-2\\0
\end{tabular}\endgroup%
\kern3pt%
\begingroup \smaller\smaller\smaller\begin{tabular}{@{}c@{}}%
5\\-1\\-1
\end{tabular}\endgroup%
\kern3pt%
\begingroup \smaller\smaller\smaller\begin{tabular}{@{}c@{}}%
15\\1\\-3
\end{tabular}\endgroup%
\kern3pt%
\begingroup \smaller\smaller\smaller\begin{tabular}{@{}c@{}}%
30\\6\\-4
\end{tabular}\endgroup%
\kern3pt%
\begingroup \smaller\smaller\smaller\begin{tabular}{@{}c@{}}%
6\\2\\0
\end{tabular}\endgroup%
{$\left.\llap{\phantom{%
\begingroup \smaller\smaller\smaller\begin{tabular}{@{}c@{}}%
0\\0\\0
\end{tabular}\endgroup%
}}\!\right]$}%
}%
\ifdim\wd\matricesbox>\halfwidth\myboxwidth=\hsize\else\myboxwidth=\halfwidth\fi
\vbox{%
\ifdim\myboxwidth=\hsize
\setbox\onelinebox=\hbox{%
\vbox{\hbox{%
$\Pi_{8,34}$ spans $L_{128.20}$%
}\hbox{%
$2|2222|222\rtimes D_{2}$%
}%
}%
\hfill\copy\matricesbox
}%
\ifdim\wd\onelinebox>\myboxwidth
\hbox to \myboxwidth{%
$\Pi_{8,34}$ spans $L_{128.20}$%
\hfil
$2|2222|222\rtimes D_{2}$%
}%
\box\matricesbox
\else
\hbox to \myboxwidth{%
\unhbox\onelinebox
}%
\fi
\else
\hbox to \myboxwidth{%
$\Pi_{8,34}$ spans $L_{128.20}$%
\hfil}%
\hbox to \myboxwidth{%
$2|2222|222\rtimes D_{2}$%
\hfil}%
\box\matricesbox
\fi
}%
\hfill\discretionary{}{}{}%
\setbox\matricesbox=\hbox{%
{$\left[\!\llap{\phantom{%
\begingroup \smaller\smaller\smaller\begin{tabular}{@{}c@{}}%
\phantom{0}\\\phantom{0}\\\phantom{0}
\end{tabular}\endgroup%
}}\right.$}%
\begingroup \smaller\smaller\smaller\begin{tabular}{@{}c@{}}%
-1/8\\\phantom{0}\\\phantom{0}
\end{tabular}\endgroup%
\kern3pt%
\begingroup \smaller\smaller\smaller\begin{tabular}{@{}c@{}}%
\phantom{0}\\10\\\phantom{0}
\end{tabular}\endgroup%
\kern3pt%
\begingroup \smaller\smaller\smaller\begin{tabular}{@{}c@{}}%
\phantom{0}\\\phantom{0}\\25/2
\end{tabular}\endgroup%
{$\left.\llap{\phantom{%
\begingroup \smaller\smaller\smaller\begin{tabular}{@{}c@{}}%
\phantom{0}\\\phantom{0}\\\phantom{0}
\end{tabular}\endgroup%
}}\!\right]$}%
{$\left[\!\llap{\phantom{%
\begingroup \smaller\smaller\smaller\begin{tabular}{@{}c@{}}%
0\\0\\0
\end{tabular}\endgroup%
}}\right.$}%
\begingroup \smaller\smaller\smaller\begin{tabular}{@{}c@{}}%
32\\-4\\0
\end{tabular}\endgroup%
\kern3pt%
\begingroup \smaller\smaller\smaller\begin{tabular}{@{}c@{}}%
50\\-5\\-3
\end{tabular}\endgroup%
\kern3pt%
\begingroup \smaller\smaller\smaller\begin{tabular}{@{}c@{}}%
40\\-2\\-4
\end{tabular}\endgroup%
\kern3pt%
\begingroup \smaller\smaller\smaller\begin{tabular}{@{}c@{}}%
10\\1\\-1
\end{tabular}\endgroup%
\kern3pt%
\begingroup \smaller\smaller\smaller\begin{tabular}{@{}c@{}}%
32\\4\\0
\end{tabular}\endgroup%
{$\left.\llap{\phantom{%
\begingroup \smaller\smaller\smaller\begin{tabular}{@{}c@{}}%
0\\0\\0
\end{tabular}\endgroup%
}}\!\right]$}%
}%
\ifdim\wd\matricesbox>\halfwidth\myboxwidth=\hsize\else\myboxwidth=\halfwidth\fi
\vbox{%
\ifdim\myboxwidth=\hsize
\setbox\onelinebox=\hbox{%
\vbox{\hbox{%
$\Pi_{8,35}$ spans $L_{10.1}$%
}\hbox{%
$22|22\infty2|2\infty\rtimes D_{2}$%
}%
}%
\hfill\copy\matricesbox
}%
\ifdim\wd\onelinebox>\myboxwidth
\hbox to \myboxwidth{%
$\Pi_{8,35}$ spans $L_{10.1}$%
\hfil
$22|22\infty2|2\infty\rtimes D_{2}$%
}%
\box\matricesbox
\else
\hbox to \myboxwidth{%
\unhbox\onelinebox
}%
\fi
\else
\hbox to \myboxwidth{%
$\Pi_{8,35}$ spans $L_{10.1}$%
\hfil}%
\hbox to \myboxwidth{%
$22|22\infty2|2\infty\rtimes D_{2}$%
\hfil}%
\box\matricesbox
\fi
}%
\hfill\discretionary{}{}{}%
\setbox\matricesbox=\hbox{%
{$\left[\!\llap{\phantom{%
\begingroup \smaller\smaller\smaller\begin{tabular}{@{}c@{}}%
\phantom{0}\\\phantom{0}\\\phantom{0}
\end{tabular}\endgroup%
}}\right.$}%
\begingroup \smaller\smaller\smaller\begin{tabular}{@{}c@{}}%
-1\\\phantom{0}\\\phantom{0}
\end{tabular}\endgroup%
\kern3pt%
\begingroup \smaller\smaller\smaller\begin{tabular}{@{}c@{}}%
\phantom{0}\\2\\\phantom{0}
\end{tabular}\endgroup%
\kern3pt%
\begingroup \smaller\smaller\smaller\begin{tabular}{@{}c@{}}%
\phantom{0}\\\phantom{0}\\4
\end{tabular}\endgroup%
{$\left.\llap{\phantom{%
\begingroup \smaller\smaller\smaller\begin{tabular}{@{}c@{}}%
\phantom{0}\\\phantom{0}\\\phantom{0}
\end{tabular}\endgroup%
}}\!\right]$}%
{$\left[\!\llap{\phantom{%
\begingroup \smaller\smaller\smaller\begin{tabular}{@{}c@{}}%
0\\0\\0
\end{tabular}\endgroup%
}}\right.$}%
\begingroup \smaller\smaller\smaller\begin{tabular}{@{}c@{}}%
1\\1\\0
\end{tabular}\endgroup%
\kern3pt%
\begingroup \smaller\smaller\smaller\begin{tabular}{@{}c@{}}%
16\\8\\6
\end{tabular}\endgroup%
\kern3pt%
\begingroup \smaller\smaller\smaller\begin{tabular}{@{}c@{}}%
8\\2\\4
\end{tabular}\endgroup%
\kern3pt%
\begingroup \smaller\smaller\smaller\begin{tabular}{@{}c@{}}%
2\\-1\\1
\end{tabular}\endgroup%
\kern3pt%
\begingroup \smaller\smaller\smaller\begin{tabular}{@{}c@{}}%
1\\-1\\0
\end{tabular}\endgroup%
{$\left.\llap{\phantom{%
\begingroup \smaller\smaller\smaller\begin{tabular}{@{}c@{}}%
0\\0\\0
\end{tabular}\endgroup%
}}\!\right]$}%
}%
\ifdim\wd\matricesbox>\halfwidth\myboxwidth=\hsize\else\myboxwidth=\halfwidth\fi
\vbox{%
\ifdim\myboxwidth=\hsize
\setbox\onelinebox=\hbox{%
\vbox{\hbox{%
$\Pi_{8,36}$ spans $L_{141.3}$%
}\hbox{%
$|22\infty2|2\infty22\rtimes D_{2}$%
}%
}%
\hfill\copy\matricesbox
}%
\ifdim\wd\onelinebox>\myboxwidth
\hbox to \myboxwidth{%
$\Pi_{8,36}$ spans $L_{141.3}$%
\hfil
$|22\infty2|2\infty22\rtimes D_{2}$%
}%
\box\matricesbox
\else
\hbox to \myboxwidth{%
\unhbox\onelinebox
}%
\fi
\else
\hbox to \myboxwidth{%
$\Pi_{8,36}$ spans $L_{141.3}$%
\hfil}%
\hbox to \myboxwidth{%
$|22\infty2|2\infty22\rtimes D_{2}$%
\hfil}%
\box\matricesbox
\fi
}%
\hfill\discretionary{}{}{}%
\setbox\matricesbox=\hbox{%
{$\left[\!\llap{\phantom{%
\begingroup \smaller\smaller\smaller\begin{tabular}{@{}c@{}}%
\phantom{0}\\\phantom{0}\\\phantom{0}
\end{tabular}\endgroup%
}}\right.$}%
\begingroup \smaller\smaller\smaller\begin{tabular}{@{}c@{}}%
-1\\\phantom{0}\\\phantom{0}
\end{tabular}\endgroup%
\kern3pt%
\begingroup \smaller\smaller\smaller\begin{tabular}{@{}c@{}}%
\phantom{0}\\2\\\phantom{0}
\end{tabular}\endgroup%
\kern3pt%
\begingroup \smaller\smaller\smaller\begin{tabular}{@{}c@{}}%
\phantom{0}\\\phantom{0}\\4
\end{tabular}\endgroup%
{$\left.\llap{\phantom{%
\begingroup \smaller\smaller\smaller\begin{tabular}{@{}c@{}}%
\phantom{0}\\\phantom{0}\\\phantom{0}
\end{tabular}\endgroup%
}}\!\right]$}%
{$\left[\!\llap{\phantom{%
\begingroup \smaller\smaller\smaller\begin{tabular}{@{}c@{}}%
0\\0\\0
\end{tabular}\endgroup%
}}\right.$}%
\begingroup \smaller\smaller\smaller\begin{tabular}{@{}c@{}}%
1\\1\\0
\end{tabular}\endgroup%
\kern3pt%
\begingroup \smaller\smaller\smaller\begin{tabular}{@{}c@{}}%
16\\8\\6
\end{tabular}\endgroup%
\kern3pt%
\begingroup \smaller\smaller\smaller\begin{tabular}{@{}c@{}}%
8\\2\\4
\end{tabular}\endgroup%
\kern3pt%
\begingroup \smaller\smaller\smaller\begin{tabular}{@{}c@{}}%
2\\-1\\1
\end{tabular}\endgroup%
\kern3pt%
\begingroup \smaller\smaller\smaller\begin{tabular}{@{}c@{}}%
8\\-6\\0
\end{tabular}\endgroup%
{$\left.\llap{\phantom{%
\begingroup \smaller\smaller\smaller\begin{tabular}{@{}c@{}}%
0\\0\\0
\end{tabular}\endgroup%
}}\!\right]$}%
}%
\ifdim\wd\matricesbox>\halfwidth\myboxwidth=\hsize\else\myboxwidth=\halfwidth\fi
\vbox{%
\ifdim\myboxwidth=\hsize
\setbox\onelinebox=\hbox{%
\vbox{\hbox{%
$\Pi_{8,37}$ spans $L_{141.3}$%
}\hbox{%
$|22\infty\infty|\infty\infty22\rtimes D_{2}$%
}%
}%
\hfill\copy\matricesbox
}%
\ifdim\wd\onelinebox>\myboxwidth
\hbox to \myboxwidth{%
$\Pi_{8,37}$ spans $L_{141.3}$%
\hfil
$|22\infty\infty|\infty\infty22\rtimes D_{2}$%
}%
\box\matricesbox
\else
\hbox to \myboxwidth{%
\unhbox\onelinebox
}%
\fi
\else
\hbox to \myboxwidth{%
$\Pi_{8,37}$ spans $L_{141.3}$%
\hfil}%
\hbox to \myboxwidth{%
$|22\infty\infty|\infty\infty22\rtimes D_{2}$%
\hfil}%
\box\matricesbox
\fi
}%
\hfill\discretionary{}{}{}%
\setbox\matricesbox=\hbox{%
{$\left[\!\llap{\phantom{%
\begingroup \smaller\smaller\smaller\begin{tabular}{@{}c@{}}%
\phantom{0}\\\phantom{0}\\\phantom{0}
\end{tabular}\endgroup%
}}\right.$}%
\begingroup \smaller\smaller\smaller\begin{tabular}{@{}c@{}}%
-1\\\phantom{0}\\\phantom{0}
\end{tabular}\endgroup%
\kern3pt%
\begingroup \smaller\smaller\smaller\begin{tabular}{@{}c@{}}%
\phantom{0}\\4\\\phantom{0}
\end{tabular}\endgroup%
\kern3pt%
\begingroup \smaller\smaller\smaller\begin{tabular}{@{}c@{}}%
\phantom{0}\\\phantom{0}\\2
\end{tabular}\endgroup%
{$\left.\llap{\phantom{%
\begingroup \smaller\smaller\smaller\begin{tabular}{@{}c@{}}%
\phantom{0}\\\phantom{0}\\\phantom{0}
\end{tabular}\endgroup%
}}\!\right]$}%
{$\left[\!\llap{\phantom{%
\begingroup \smaller\smaller\smaller\begin{tabular}{@{}c@{}}%
0\\0\\0
\end{tabular}\endgroup%
}}\right.$}%
\begingroup \smaller\smaller\smaller\begin{tabular}{@{}c@{}}%
8\\-4\\2
\end{tabular}\endgroup%
\kern3pt%
\begingroup \smaller\smaller\smaller\begin{tabular}{@{}c@{}}%
16\\-6\\8
\end{tabular}\endgroup%
\kern3pt%
\begingroup \smaller\smaller\smaller\begin{tabular}{@{}c@{}}%
1\\0\\1
\end{tabular}\endgroup%
\kern3pt%
\begingroup \smaller\smaller\smaller\begin{tabular}{@{}c@{}}%
2\\1\\1
\end{tabular}\endgroup%
{$\left.\llap{\phantom{%
\begingroup \smaller\smaller\smaller\begin{tabular}{@{}c@{}}%
0\\0\\0
\end{tabular}\endgroup%
}}\!\right]$}%
}%
\ifdim\wd\matricesbox>\halfwidth\myboxwidth=\hsize\else\myboxwidth=\halfwidth\fi
\vbox{%
\ifdim\myboxwidth=\hsize
\setbox\onelinebox=\hbox{%
\vbox{\hbox{%
$\Pi_{8,38}$ spans $L_{141.3}$%
}\hbox{%
$22\slashinfty222\slashinfty2\rtimes D_{2}$%
}%
}%
\hfill\copy\matricesbox
}%
\ifdim\wd\onelinebox>\myboxwidth
\hbox to \myboxwidth{%
$\Pi_{8,38}$ spans $L_{141.3}$%
\hfil
$22\slashinfty222\slashinfty2\rtimes D_{2}$%
}%
\box\matricesbox
\else
\hbox to \myboxwidth{%
\unhbox\onelinebox
}%
\fi
\else
\hbox to \myboxwidth{%
$\Pi_{8,38}$ spans $L_{141.3}$%
\hfil}%
\hbox to \myboxwidth{%
$22\slashinfty222\slashinfty2\rtimes D_{2}$%
\hfil}%
\box\matricesbox
\fi
}%
\hfill\discretionary{}{}{}%
\setbox\matricesbox=\hbox{%
{$\left[\!\llap{\phantom{%
\begingroup \smaller\smaller\smaller\begin{tabular}{@{}c@{}}%
\phantom{0}\\\phantom{0}\\\phantom{0}
\end{tabular}\endgroup%
}}\right.$}%
\begingroup \smaller\smaller\smaller\begin{tabular}{@{}c@{}}%
-1/2\\\phantom{0}\\\phantom{0}
\end{tabular}\endgroup%
\kern3pt%
\begingroup \smaller\smaller\smaller\begin{tabular}{@{}c@{}}%
\phantom{0}\\3\\\phantom{0}
\end{tabular}\endgroup%
\kern3pt%
\begingroup \smaller\smaller\smaller\begin{tabular}{@{}c@{}}%
\phantom{0}\\\phantom{0}\\9/2
\end{tabular}\endgroup%
{$\left.\llap{\phantom{%
\begingroup \smaller\smaller\smaller\begin{tabular}{@{}c@{}}%
\phantom{0}\\\phantom{0}\\\phantom{0}
\end{tabular}\endgroup%
}}\!\right]$}%
{$\left[\!\llap{\phantom{%
\begingroup \smaller\smaller\smaller\begin{tabular}{@{}c@{}}%
0\\0\\0
\end{tabular}\endgroup%
}}\right.$}%
\begingroup \smaller\smaller\smaller\begin{tabular}{@{}c@{}}%
4\\-2\\0
\end{tabular}\endgroup%
\kern3pt%
\begingroup \smaller\smaller\smaller\begin{tabular}{@{}c@{}}%
18\\-6\\4
\end{tabular}\endgroup%
\kern3pt%
\begingroup \smaller\smaller\smaller\begin{tabular}{@{}c@{}}%
12\\-2\\4
\end{tabular}\endgroup%
\kern3pt%
\begingroup \smaller\smaller\smaller\begin{tabular}{@{}c@{}}%
3\\1\\1
\end{tabular}\endgroup%
\kern3pt%
\begingroup \smaller\smaller\smaller\begin{tabular}{@{}c@{}}%
4\\2\\0
\end{tabular}\endgroup%
{$\left.\llap{\phantom{%
\begingroup \smaller\smaller\smaller\begin{tabular}{@{}c@{}}%
0\\0\\0
\end{tabular}\endgroup%
}}\!\right]$}%
}%
\ifdim\wd\matricesbox>\halfwidth\myboxwidth=\hsize\else\myboxwidth=\halfwidth\fi
\vbox{%
\ifdim\myboxwidth=\hsize
\setbox\onelinebox=\hbox{%
\vbox{\hbox{%
$\Pi_{8,39}$ spans $L_{4.12}$%
}\hbox{%
$|22\infty2|2\infty22\rtimes D_{2}$%
}%
}%
\hfill\copy\matricesbox
}%
\ifdim\wd\onelinebox>\myboxwidth
\hbox to \myboxwidth{%
$\Pi_{8,39}$ spans $L_{4.12}$%
\hfil
$|22\infty2|2\infty22\rtimes D_{2}$%
}%
\box\matricesbox
\else
\hbox to \myboxwidth{%
\unhbox\onelinebox
}%
\fi
\else
\hbox to \myboxwidth{%
$\Pi_{8,39}$ spans $L_{4.12}$%
\hfil}%
\hbox to \myboxwidth{%
$|22\infty2|2\infty22\rtimes D_{2}$%
\hfil}%
\box\matricesbox
\fi
}%
\hfill\discretionary{}{}{}%
\setbox\matricesbox=\hbox{%
{$\left[\!\llap{\phantom{%
\begingroup \smaller\smaller\smaller\begin{tabular}{@{}c@{}}%
\phantom{0}\\\phantom{0}\\\phantom{0}
\end{tabular}\endgroup%
}}\right.$}%
\begingroup \smaller\smaller\smaller\begin{tabular}{@{}c@{}}%
-1/2\\\phantom{0}\\\phantom{0}
\end{tabular}\endgroup%
\kern3pt%
\begingroup \smaller\smaller\smaller\begin{tabular}{@{}c@{}}%
\phantom{0}\\9/2\\\phantom{0}
\end{tabular}\endgroup%
\kern3pt%
\begingroup \smaller\smaller\smaller\begin{tabular}{@{}c@{}}%
\phantom{0}\\\phantom{0}\\3
\end{tabular}\endgroup%
{$\left.\llap{\phantom{%
\begingroup \smaller\smaller\smaller\begin{tabular}{@{}c@{}}%
\phantom{0}\\\phantom{0}\\\phantom{0}
\end{tabular}\endgroup%
}}\!\right]$}%
{$\left[\!\llap{\phantom{%
\begingroup \smaller\smaller\smaller\begin{tabular}{@{}c@{}}%
0\\0\\0
\end{tabular}\endgroup%
}}\right.$}%
\begingroup \smaller\smaller\smaller\begin{tabular}{@{}c@{}}%
12\\4\\2
\end{tabular}\endgroup%
\kern3pt%
\begingroup \smaller\smaller\smaller\begin{tabular}{@{}c@{}}%
18\\4\\6
\end{tabular}\endgroup%
\kern3pt%
\begingroup \smaller\smaller\smaller\begin{tabular}{@{}c@{}}%
4\\0\\2
\end{tabular}\endgroup%
\kern3pt%
\begingroup \smaller\smaller\smaller\begin{tabular}{@{}c@{}}%
3\\-1\\1
\end{tabular}\endgroup%
{$\left.\llap{\phantom{%
\begingroup \smaller\smaller\smaller\begin{tabular}{@{}c@{}}%
0\\0\\0
\end{tabular}\endgroup%
}}\!\right]$}%
}%
\ifdim\wd\matricesbox>\halfwidth\myboxwidth=\hsize\else\myboxwidth=\halfwidth\fi
\vbox{%
\ifdim\myboxwidth=\hsize
\setbox\onelinebox=\hbox{%
\vbox{\hbox{%
$\Pi_{8,40}$ spans $L_{4.12}$%
}\hbox{%
$22\slashinfty222\slashinfty2\rtimes D_{2}$%
}%
}%
\hfill\copy\matricesbox
}%
\ifdim\wd\onelinebox>\myboxwidth
\hbox to \myboxwidth{%
$\Pi_{8,40}$ spans $L_{4.12}$%
\hfil
$22\slashinfty222\slashinfty2\rtimes D_{2}$%
}%
\box\matricesbox
\else
\hbox to \myboxwidth{%
\unhbox\onelinebox
}%
\fi
\else
\hbox to \myboxwidth{%
$\Pi_{8,40}$ spans $L_{4.12}$%
\hfil}%
\hbox to \myboxwidth{%
$22\slashinfty222\slashinfty2\rtimes D_{2}$%
\hfil}%
\box\matricesbox
\fi
}%
\hfill\discretionary{}{}{}%
\setbox\matricesbox=\hbox{%
{$\left[\!\llap{\phantom{%
\begingroup \smaller\smaller\smaller\begin{tabular}{@{}c@{}}%
\phantom{0}\\\phantom{0}\\\phantom{0}
\end{tabular}\endgroup%
}}\right.$}%
\begingroup \smaller\smaller\smaller\begin{tabular}{@{}c@{}}%
-1/4\\\phantom{0}\\\phantom{0}
\end{tabular}\endgroup%
\kern3pt%
\begingroup \smaller\smaller\smaller\begin{tabular}{@{}c@{}}%
\phantom{0}\\25/4\\\phantom{0}
\end{tabular}\endgroup%
\kern3pt%
\begingroup \smaller\smaller\smaller\begin{tabular}{@{}c@{}}%
\phantom{0}\\\phantom{0}\\5
\end{tabular}\endgroup%
{$\left.\llap{\phantom{%
\begingroup \smaller\smaller\smaller\begin{tabular}{@{}c@{}}%
\phantom{0}\\\phantom{0}\\\phantom{0}
\end{tabular}\endgroup%
}}\!\right]$}%
{$\left[\!\llap{\phantom{%
\begingroup \smaller\smaller\smaller\begin{tabular}{@{}c@{}}%
0\\0\\0
\end{tabular}\endgroup%
}}\right.$}%
\begingroup \smaller\smaller\smaller\begin{tabular}{@{}c@{}}%
20\\4\\-2
\end{tabular}\endgroup%
\kern3pt%
\begingroup \smaller\smaller\smaller\begin{tabular}{@{}c@{}}%
25\\3\\-5
\end{tabular}\endgroup%
\kern3pt%
\begingroup \smaller\smaller\smaller\begin{tabular}{@{}c@{}}%
16\\0\\-4
\end{tabular}\endgroup%
\kern3pt%
\begingroup \smaller\smaller\smaller\begin{tabular}{@{}c@{}}%
5\\-1\\-1
\end{tabular}\endgroup%
{$\left.\llap{\phantom{%
\begingroup \smaller\smaller\smaller\begin{tabular}{@{}c@{}}%
0\\0\\0
\end{tabular}\endgroup%
}}\!\right]$}%
}%
\ifdim\wd\matricesbox>\halfwidth\myboxwidth=\hsize\else\myboxwidth=\halfwidth\fi
\vbox{%
\ifdim\myboxwidth=\hsize
\setbox\onelinebox=\hbox{%
\vbox{\hbox{%
$\Pi_{8,41}$ spans $L_{149.22}$%
}\hbox{%
$22\slashinfty222\slashinfty2\rtimes D_{2}$%
}%
}%
\hfill\copy\matricesbox
}%
\ifdim\wd\onelinebox>\myboxwidth
\hbox to \myboxwidth{%
$\Pi_{8,41}$ spans $L_{149.22}$%
\hfil
$22\slashinfty222\slashinfty2\rtimes D_{2}$%
}%
\box\matricesbox
\else
\hbox to \myboxwidth{%
\unhbox\onelinebox
}%
\fi
\else
\hbox to \myboxwidth{%
$\Pi_{8,41}$ spans $L_{149.22}$%
\hfil}%
\hbox to \myboxwidth{%
$22\slashinfty222\slashinfty2\rtimes D_{2}$%
\hfil}%
\box\matricesbox
\fi
}%
\hfill\discretionary{}{}{}%
\setbox\matricesbox=\hbox{%
{$\left[\!\llap{\phantom{%
\begingroup \smaller\smaller\smaller\begin{tabular}{@{}c@{}}%
\phantom{0}\\\phantom{0}\\\phantom{0}
\end{tabular}\endgroup%
}}\right.$}%
\begingroup \smaller\smaller\smaller\begin{tabular}{@{}c@{}}%
-1/8\\\phantom{0}\\\phantom{0}
\end{tabular}\endgroup%
\kern3pt%
\begingroup \smaller\smaller\smaller\begin{tabular}{@{}c@{}}%
\phantom{0}\\21/2\\\phantom{0}
\end{tabular}\endgroup%
\kern3pt%
\begingroup \smaller\smaller\smaller\begin{tabular}{@{}c@{}}%
\phantom{0}\\\phantom{0}\\28
\end{tabular}\endgroup%
{$\left.\llap{\phantom{%
\begingroup \smaller\smaller\smaller\begin{tabular}{@{}c@{}}%
\phantom{0}\\\phantom{0}\\\phantom{0}
\end{tabular}\endgroup%
}}\!\right]$}%
{$\left[\!\llap{\phantom{%
\begingroup \smaller\smaller\smaller\begin{tabular}{@{}c@{}}%
0\\0\\0
\end{tabular}\endgroup%
}}\right.$}%
\begingroup \smaller\smaller\smaller\begin{tabular}{@{}c@{}}%
6\\1\\0
\end{tabular}\endgroup%
\kern3pt%
\begingroup \smaller\smaller\smaller\begin{tabular}{@{}c@{}}%
112\\8\\-6
\end{tabular}\endgroup%
\kern3pt%
\begingroup \smaller\smaller\smaller\begin{tabular}{@{}c@{}}%
42\\1\\-3
\end{tabular}\endgroup%
\kern3pt%
\begingroup \smaller\smaller\smaller\begin{tabular}{@{}c@{}}%
14\\-1\\-1
\end{tabular}\endgroup%
\kern3pt%
\begingroup \smaller\smaller\smaller\begin{tabular}{@{}c@{}}%
42\\-5\\0
\end{tabular}\endgroup%
{$\left.\llap{\phantom{%
\begingroup \smaller\smaller\smaller\begin{tabular}{@{}c@{}}%
0\\0\\0
\end{tabular}\endgroup%
}}\!\right]$}%
}%
\ifdim\wd\matricesbox>\halfwidth\myboxwidth=\hsize\else\myboxwidth=\halfwidth\fi
\vbox{%
\ifdim\myboxwidth=\hsize
\setbox\onelinebox=\hbox{%
\vbox{\hbox{%
$\Pi_{8,42}$ spans $L_{259.3}$%
}\hbox{%
$|2226|6222\rtimes D_{2}$%
}%
}%
\hfill\copy\matricesbox
}%
\ifdim\wd\onelinebox>\myboxwidth
\hbox to \myboxwidth{%
$\Pi_{8,42}$ spans $L_{259.3}$%
\hfil
$|2226|6222\rtimes D_{2}$%
}%
\box\matricesbox
\else
\hbox to \myboxwidth{%
\unhbox\onelinebox
}%
\fi
\else
\hbox to \myboxwidth{%
$\Pi_{8,42}$ spans $L_{259.3}$%
\hfil}%
\hbox to \myboxwidth{%
$|2226|6222\rtimes D_{2}$%
\hfil}%
\box\matricesbox
\fi
}%
\hfill\discretionary{}{}{}%
\setbox\matricesbox=\hbox{%
{$\left[\!\llap{\phantom{%
\begingroup \smaller\smaller\smaller\begin{tabular}{@{}c@{}}%
\phantom{0}\\\phantom{0}\\\phantom{0}
\end{tabular}\endgroup%
}}\right.$}%
\begingroup \smaller\smaller\smaller\begin{tabular}{@{}c@{}}%
-1\\\phantom{0}\\\phantom{0}
\end{tabular}\endgroup%
\kern3pt%
\begingroup \smaller\smaller\smaller\begin{tabular}{@{}c@{}}%
\phantom{0}\\3/2\\\phantom{0}
\end{tabular}\endgroup%
\kern3pt%
\begingroup \smaller\smaller\smaller\begin{tabular}{@{}c@{}}%
\phantom{0}\\\phantom{0}\\1/2
\end{tabular}\endgroup%
{$\left.\llap{\phantom{%
\begingroup \smaller\smaller\smaller\begin{tabular}{@{}c@{}}%
\phantom{0}\\\phantom{0}\\\phantom{0}
\end{tabular}\endgroup%
}}\!\right]$}%
{$\left[\!\llap{\phantom{%
\begingroup \smaller\smaller\smaller\begin{tabular}{@{}c@{}}%
0\\0\\0
\end{tabular}\endgroup%
}}\right.$}%
\begingroup \smaller\smaller\smaller\begin{tabular}{@{}c@{}}%
6\\-5\\-3
\end{tabular}\endgroup%
\kern3pt%
\begingroup \smaller\smaller\smaller\begin{tabular}{@{}c@{}}%
2\\-1\\-3
\end{tabular}\endgroup%
\kern3pt%
\begingroup \smaller\smaller\smaller\begin{tabular}{@{}c@{}}%
6\\1\\-9
\end{tabular}\endgroup%
\kern3pt%
\begingroup \smaller\smaller\smaller\begin{tabular}{@{}c@{}}%
1\\1\\-1
\end{tabular}\endgroup%
{$\left.\llap{\phantom{%
\begingroup \smaller\smaller\smaller\begin{tabular}{@{}c@{}}%
0\\0\\0
\end{tabular}\endgroup%
}}\!\right]$}%
}%
\ifdim\wd\matricesbox>\halfwidth\myboxwidth=\hsize\else\myboxwidth=\halfwidth\fi
\vbox{%
\ifdim\myboxwidth=\hsize
\setbox\onelinebox=\hbox{%
\vbox{\hbox{%
$\Pi_{8,43}$ spans $L_{3.2}$%
}\hbox{%
$222\slashthree222\slashtwo\rtimes D_{2}$%
}%
}%
\hfill\copy\matricesbox
}%
\ifdim\wd\onelinebox>\myboxwidth
\hbox to \myboxwidth{%
$\Pi_{8,43}$ spans $L_{3.2}$%
\hfil
$222\slashthree222\slashtwo\rtimes D_{2}$%
}%
\box\matricesbox
\else
\hbox to \myboxwidth{%
\unhbox\onelinebox
}%
\fi
\else
\hbox to \myboxwidth{%
$\Pi_{8,43}$ spans $L_{3.2}$%
\hfil}%
\hbox to \myboxwidth{%
$222\slashthree222\slashtwo\rtimes D_{2}$%
\hfil}%
\box\matricesbox
\fi
}%
\hfill\discretionary{}{}{}%
\setbox\matricesbox=\hbox{%
{$\left[\!\llap{\phantom{%
\begingroup \smaller\smaller\smaller\begin{tabular}{@{}c@{}}%
\phantom{0}\\\phantom{0}\\\phantom{0}
\end{tabular}\endgroup%
}}\right.$}%
\begingroup \smaller\smaller\smaller\begin{tabular}{@{}c@{}}%
-1\\\phantom{0}\\\phantom{0}
\end{tabular}\endgroup%
\kern3pt%
\begingroup \smaller\smaller\smaller\begin{tabular}{@{}c@{}}%
\phantom{0}\\3/2\\\phantom{0}
\end{tabular}\endgroup%
\kern3pt%
\begingroup \smaller\smaller\smaller\begin{tabular}{@{}c@{}}%
\phantom{0}\\\phantom{0}\\9/2
\end{tabular}\endgroup%
{$\left.\llap{\phantom{%
\begingroup \smaller\smaller\smaller\begin{tabular}{@{}c@{}}%
\phantom{0}\\\phantom{0}\\\phantom{0}
\end{tabular}\endgroup%
}}\!\right]$}%
{$\left[\!\llap{\phantom{%
\begingroup \smaller\smaller\smaller\begin{tabular}{@{}c@{}}%
0\\0\\0
\end{tabular}\endgroup%
}}\right.$}%
\begingroup \smaller\smaller\smaller\begin{tabular}{@{}c@{}}%
2\\-2\\0
\end{tabular}\endgroup%
\kern3pt%
\begingroup \smaller\smaller\smaller\begin{tabular}{@{}c@{}}%
2\\-1\\-1
\end{tabular}\endgroup%
\kern3pt%
\begingroup \smaller\smaller\smaller\begin{tabular}{@{}c@{}}%
6\\1\\-3
\end{tabular}\endgroup%
\kern3pt%
\begingroup \smaller\smaller\smaller\begin{tabular}{@{}c@{}}%
6\\4\\-2
\end{tabular}\endgroup%
\kern3pt%
\begingroup \smaller\smaller\smaller\begin{tabular}{@{}c@{}}%
2\\2\\0
\end{tabular}\endgroup%
{$\left.\llap{\phantom{%
\begingroup \smaller\smaller\smaller\begin{tabular}{@{}c@{}}%
0\\0\\0
\end{tabular}\endgroup%
}}\!\right]$}%
}%
\ifdim\wd\matricesbox>\halfwidth\myboxwidth=\hsize\else\myboxwidth=\halfwidth\fi
\vbox{%
\ifdim\myboxwidth=\hsize
\setbox\onelinebox=\hbox{%
\vbox{\hbox{%
$\Pi_{8,44}$ spans $L_{155.1}$%
}\hbox{%
$3|3232|232\rtimes D_{2}$%
}%
}%
\hfill\copy\matricesbox
}%
\ifdim\wd\onelinebox>\myboxwidth
\hbox to \myboxwidth{%
$\Pi_{8,44}$ spans $L_{155.1}$%
\hfil
$3|3232|232\rtimes D_{2}$%
}%
\box\matricesbox
\else
\hbox to \myboxwidth{%
\unhbox\onelinebox
}%
\fi
\else
\hbox to \myboxwidth{%
$\Pi_{8,44}$ spans $L_{155.1}$%
\hfil}%
\hbox to \myboxwidth{%
$3|3232|232\rtimes D_{2}$%
\hfil}%
\box\matricesbox
\fi
}%
\hfill\discretionary{}{}{}%
\setbox\matricesbox=\hbox{%
{$\left[\!\llap{\phantom{%
\begingroup \smaller\smaller\smaller\begin{tabular}{@{}c@{}}%
\phantom{0}\\\phantom{0}\\\phantom{0}
\end{tabular}\endgroup%
}}\right.$}%
\begingroup \smaller\smaller\smaller\begin{tabular}{@{}c@{}}%
-3\\\phantom{0}\\\phantom{0}
\end{tabular}\endgroup%
\kern3pt%
\begingroup \smaller\smaller\smaller\begin{tabular}{@{}c@{}}%
\phantom{0}\\3\\\phantom{0}
\end{tabular}\endgroup%
\kern3pt%
\begingroup \smaller\smaller\smaller\begin{tabular}{@{}c@{}}%
\phantom{0}\\\phantom{0}\\1
\end{tabular}\endgroup%
{$\left.\llap{\phantom{%
\begingroup \smaller\smaller\smaller\begin{tabular}{@{}c@{}}%
\phantom{0}\\\phantom{0}\\\phantom{0}
\end{tabular}\endgroup%
}}\!\right]$}%
{$\left[\!\llap{\phantom{%
\begingroup \smaller\smaller\smaller\begin{tabular}{@{}c@{}}%
0\\0\\0
\end{tabular}\endgroup%
}}\right.$}%
\begingroup \smaller\smaller\smaller\begin{tabular}{@{}c@{}}%
1\\-1\\1
\end{tabular}\endgroup%
\kern3pt%
\begingroup \smaller\smaller\smaller\begin{tabular}{@{}c@{}}%
1\\0\\2
\end{tabular}\endgroup%
\kern3pt%
\begingroup \smaller\smaller\smaller\begin{tabular}{@{}c@{}}%
4\\3\\5
\end{tabular}\endgroup%
\kern3pt%
\begingroup \smaller\smaller\smaller\begin{tabular}{@{}c@{}}%
4\\4\\2
\end{tabular}\endgroup%
{$\left.\llap{\phantom{%
\begingroup \smaller\smaller\smaller\begin{tabular}{@{}c@{}}%
0\\0\\0
\end{tabular}\endgroup%
}}\!\right]$}%
}%
\ifdim\wd\matricesbox>\halfwidth\myboxwidth=\hsize\else\myboxwidth=\halfwidth\fi
\vbox{%
\ifdim\myboxwidth=\hsize
\setbox\onelinebox=\hbox{%
\vbox{\hbox{%
$\Pi_{8,45}=\hbox{GN}_{48}$ spans $L_{7.7}$%
}\hbox{%
$3\infty\infty\slashinfty\infty\infty3\slashinfty\rtimes D_{2}$%
}%
}%
\hfill\copy\matricesbox
}%
\ifdim\wd\onelinebox>\myboxwidth
\hbox to \myboxwidth{%
$\Pi_{8,45}=\hbox{GN}_{48}$ spans $L_{7.7}$%
\hfil
$3\infty\infty\slashinfty\infty\infty3\slashinfty\rtimes D_{2}$%
}%
\box\matricesbox
\else
\hbox to \myboxwidth{%
\unhbox\onelinebox
}%
\fi
\else
\hbox to \myboxwidth{%
$\Pi_{8,45}=\hbox{GN}_{48}$ spans $L_{7.7}$%
\hfil}%
\hbox to \myboxwidth{%
$3\infty\infty\slashinfty\infty\infty3\slashinfty\rtimes D_{2}$%
\hfil}%
\box\matricesbox
\fi
}%
\hfill\discretionary{}{}{}%
\setbox\matricesbox=\hbox{%
{$\left[\!\llap{\phantom{%
\begingroup \smaller\smaller\smaller\begin{tabular}{@{}c@{}}%
\phantom{0}\\\phantom{0}\\\phantom{0}
\end{tabular}\endgroup%
}}\right.$}%
\begingroup \smaller\smaller\smaller\begin{tabular}{@{}c@{}}%
-3\\\phantom{0}\\\phantom{0}
\end{tabular}\endgroup%
\kern3pt%
\begingroup \smaller\smaller\smaller\begin{tabular}{@{}c@{}}%
\phantom{0}\\1\\\phantom{0}
\end{tabular}\endgroup%
\kern3pt%
\begingroup \smaller\smaller\smaller\begin{tabular}{@{}c@{}}%
\phantom{0}\\\phantom{0}\\3
\end{tabular}\endgroup%
{$\left.\llap{\phantom{%
\begingroup \smaller\smaller\smaller\begin{tabular}{@{}c@{}}%
\phantom{0}\\\phantom{0}\\\phantom{0}
\end{tabular}\endgroup%
}}\!\right]$}%
{$\left[\!\llap{\phantom{%
\begingroup \smaller\smaller\smaller\begin{tabular}{@{}c@{}}%
0\\0\\0
\end{tabular}\endgroup%
}}\right.$}%
\begingroup \smaller\smaller\smaller\begin{tabular}{@{}c@{}}%
1\\2\\0
\end{tabular}\endgroup%
\kern3pt%
\begingroup \smaller\smaller\smaller\begin{tabular}{@{}c@{}}%
1\\1\\1
\end{tabular}\endgroup%
\kern3pt%
\begingroup \smaller\smaller\smaller\begin{tabular}{@{}c@{}}%
4\\-2\\4
\end{tabular}\endgroup%
\kern3pt%
\begingroup \smaller\smaller\smaller\begin{tabular}{@{}c@{}}%
4\\-5\\3
\end{tabular}\endgroup%
\kern3pt%
\begingroup \smaller\smaller\smaller\begin{tabular}{@{}c@{}}%
1\\-2\\0
\end{tabular}\endgroup%
{$\left.\llap{\phantom{%
\begingroup \smaller\smaller\smaller\begin{tabular}{@{}c@{}}%
0\\0\\0
\end{tabular}\endgroup%
}}\!\right]$}%
}%
\ifdim\wd\matricesbox>\halfwidth\myboxwidth=\hsize\else\myboxwidth=\halfwidth\fi
\vbox{%
\ifdim\myboxwidth=\hsize
\setbox\onelinebox=\hbox{%
\vbox{\hbox{%
$\Pi_{8,46}=\hbox{GN}_{49}$ spans $L_{7.7}$%
}\hbox{%
$3\infty\infty|\infty\infty3\infty|\infty\rtimes D_{2}$%
}%
}%
\hfill\copy\matricesbox
}%
\ifdim\wd\onelinebox>\myboxwidth
\hbox to \myboxwidth{%
$\Pi_{8,46}=\hbox{GN}_{49}$ spans $L_{7.7}$%
\hfil
$3\infty\infty|\infty\infty3\infty|\infty\rtimes D_{2}$%
}%
\box\matricesbox
\else
\hbox to \myboxwidth{%
\unhbox\onelinebox
}%
\fi
\else
\hbox to \myboxwidth{%
$\Pi_{8,46}=\hbox{GN}_{49}$ spans $L_{7.7}$%
\hfil}%
\hbox to \myboxwidth{%
$3\infty\infty|\infty\infty3\infty|\infty\rtimes D_{2}$%
\hfil}%
\box\matricesbox
\fi
}%
\hfill\discretionary{}{}{}%
\setbox\matricesbox=\hbox{%
{$\left[\!\llap{\phantom{%
\begingroup \smaller\smaller\smaller\begin{tabular}{@{}c@{}}%
\phantom{0}\\\phantom{0}\\\phantom{0}
\end{tabular}\endgroup%
}}\right.$}%
\begingroup \smaller\smaller\smaller\begin{tabular}{@{}c@{}}%
-1/2\\\phantom{0}\\\phantom{0}
\end{tabular}\endgroup%
\kern3pt%
\begingroup \smaller\smaller\smaller\begin{tabular}{@{}c@{}}%
\phantom{0}\\1/2\\\phantom{0}
\end{tabular}\endgroup%
\kern3pt%
\begingroup \smaller\smaller\smaller\begin{tabular}{@{}c@{}}%
\phantom{0}\\\phantom{0}\\7/2
\end{tabular}\endgroup%
{$\left.\llap{\phantom{%
\begingroup \smaller\smaller\smaller\begin{tabular}{@{}c@{}}%
\phantom{0}\\\phantom{0}\\\phantom{0}
\end{tabular}\endgroup%
}}\!\right]$}%
{$\left[\!\llap{\phantom{%
\begingroup \smaller\smaller\smaller\begin{tabular}{@{}c@{}}%
0\\0\\0
\end{tabular}\endgroup%
}}\right.$}%
\begingroup \smaller\smaller\smaller\begin{tabular}{@{}c@{}}%
14\\-14\\-2
\end{tabular}\endgroup%
\kern3pt%
\begingroup \smaller\smaller\smaller\begin{tabular}{@{}c@{}}%
14\\-7\\-5
\end{tabular}\endgroup%
\kern3pt%
\begingroup \smaller\smaller\smaller\begin{tabular}{@{}c@{}}%
2\\1\\-1
\end{tabular}\endgroup%
\kern3pt%
\begingroup \smaller\smaller\smaller\begin{tabular}{@{}c@{}}%
14\\14\\-2
\end{tabular}\endgroup%
{$\left.\llap{\phantom{%
\begingroup \smaller\smaller\smaller\begin{tabular}{@{}c@{}}%
0\\0\\0
\end{tabular}\endgroup%
}}\!\right]$}%
}%
\ifdim\wd\matricesbox>\halfwidth\myboxwidth=\hsize\else\myboxwidth=\halfwidth\fi
\vbox{%
\ifdim\myboxwidth=\hsize
\setbox\onelinebox=\hbox{%
\vbox{\hbox{%
$\Pi_{8,47}$ spans $L_{300.3}$%
}\hbox{%
$\infty\slashinfty\infty22\slashinfty22\rtimes D_{2}$%
}%
}%
\hfill\copy\matricesbox
}%
\ifdim\wd\onelinebox>\myboxwidth
\hbox to \myboxwidth{%
$\Pi_{8,47}$ spans $L_{300.3}$%
\hfil
$\infty\slashinfty\infty22\slashinfty22\rtimes D_{2}$%
}%
\box\matricesbox
\else
\hbox to \myboxwidth{%
\unhbox\onelinebox
}%
\fi
\else
\hbox to \myboxwidth{%
$\Pi_{8,47}$ spans $L_{300.3}$%
\hfil}%
\hbox to \myboxwidth{%
$\infty\slashinfty\infty22\slashinfty22\rtimes D_{2}$%
\hfil}%
\box\matricesbox
\fi
}%
\hfill\discretionary{}{}{}%
\setbox\matricesbox=\hbox{%
{$\left[\!\llap{\phantom{%
\begingroup \smaller\smaller\smaller\begin{tabular}{@{}c@{}}%
\phantom{0}\\\phantom{0}\\\phantom{0}
\end{tabular}\endgroup%
}}\right.$}%
\begingroup \smaller\smaller\smaller\begin{tabular}{@{}c@{}}%
-1/8\\\phantom{0}\\\phantom{0}
\end{tabular}\endgroup%
\kern3pt%
\begingroup \smaller\smaller\smaller\begin{tabular}{@{}c@{}}%
\phantom{0}\\21/2\\\phantom{0}
\end{tabular}\endgroup%
\kern3pt%
\begingroup \smaller\smaller\smaller\begin{tabular}{@{}c@{}}%
\phantom{0}\\\phantom{0}\\28
\end{tabular}\endgroup%
{$\left.\llap{\phantom{%
\begingroup \smaller\smaller\smaller\begin{tabular}{@{}c@{}}%
\phantom{0}\\\phantom{0}\\\phantom{0}
\end{tabular}\endgroup%
}}\!\right]$}%
{$\left[\!\llap{\phantom{%
\begingroup \smaller\smaller\smaller\begin{tabular}{@{}c@{}}%
0\\0\\0
\end{tabular}\endgroup%
}}\right.$}%
\begingroup \smaller\smaller\smaller\begin{tabular}{@{}c@{}}%
6\\-1\\0
\end{tabular}\endgroup%
\kern3pt%
\begingroup \smaller\smaller\smaller\begin{tabular}{@{}c@{}}%
112\\-8\\-6
\end{tabular}\endgroup%
\kern3pt%
\begingroup \smaller\smaller\smaller\begin{tabular}{@{}c@{}}%
42\\-1\\-3
\end{tabular}\endgroup%
\kern3pt%
\begingroup \smaller\smaller\smaller\begin{tabular}{@{}c@{}}%
14\\1\\-1
\end{tabular}\endgroup%
{$\left.\llap{\phantom{%
\begingroup \smaller\smaller\smaller\begin{tabular}{@{}c@{}}%
0\\0\\0
\end{tabular}\endgroup%
}}\!\right]$}%
}%
\ifdim\wd\matricesbox>\halfwidth\myboxwidth=\hsize\else\myboxwidth=\halfwidth\fi
\vbox{%
\ifdim\myboxwidth=\hsize
\setbox\onelinebox=\hbox{%
\vbox{\hbox{%
$\Pi_{8,48}$ spans $L_{22.8}$%
}\hbox{%
$22222222\rtimes C_{2}$%
}%
}%
\hfill\copy\matricesbox
}%
\ifdim\wd\onelinebox>\myboxwidth
\hbox to \myboxwidth{%
$\Pi_{8,48}$ spans $L_{22.8}$%
\hfil
$22222222\rtimes C_{2}$%
}%
\box\matricesbox
\else
\hbox to \myboxwidth{%
\unhbox\onelinebox
}%
\fi
\else
\hbox to \myboxwidth{%
$\Pi_{8,48}$ spans $L_{22.8}$%
\hfil}%
\hbox to \myboxwidth{%
$22222222\rtimes C_{2}$%
\hfil}%
\box\matricesbox
\fi
}%
\hfill\discretionary{}{}{}%
\setbox\matricesbox=\hbox{%
{$\left[\!\llap{\phantom{%
\begingroup \smaller\smaller\smaller\begin{tabular}{@{}c@{}}%
\phantom{0}\\\phantom{0}\\\phantom{0}
\end{tabular}\endgroup%
}}\right.$}%
\begingroup \smaller\smaller\smaller\begin{tabular}{@{}c@{}}%
-1\\\phantom{0}\\\phantom{0}
\end{tabular}\endgroup%
\kern3pt%
\begingroup \smaller\smaller\smaller\begin{tabular}{@{}c@{}}%
\phantom{0}\\2\\\phantom{0}
\end{tabular}\endgroup%
\kern3pt%
\begingroup \smaller\smaller\smaller\begin{tabular}{@{}c@{}}%
\phantom{0}\\\phantom{0}\\4
\end{tabular}\endgroup%
{$\left.\llap{\phantom{%
\begingroup \smaller\smaller\smaller\begin{tabular}{@{}c@{}}%
\phantom{0}\\\phantom{0}\\\phantom{0}
\end{tabular}\endgroup%
}}\!\right]$}%
{$\left[\!\llap{\phantom{%
\begingroup \smaller\smaller\smaller\begin{tabular}{@{}c@{}}%
0\\0\\0
\end{tabular}\endgroup%
}}\right.$}%
\begingroup \smaller\smaller\smaller\begin{tabular}{@{}c@{}}%
1\\1\\0
\end{tabular}\endgroup%
\kern3pt%
\begingroup \smaller\smaller\smaller\begin{tabular}{@{}c@{}}%
16\\8\\-6
\end{tabular}\endgroup%
\kern3pt%
\begingroup \smaller\smaller\smaller\begin{tabular}{@{}c@{}}%
8\\2\\-4
\end{tabular}\endgroup%
\kern3pt%
\begingroup \smaller\smaller\smaller\begin{tabular}{@{}c@{}}%
2\\-1\\-1
\end{tabular}\endgroup%
{$\left.\llap{\phantom{%
\begingroup \smaller\smaller\smaller\begin{tabular}{@{}c@{}}%
0\\0\\0
\end{tabular}\endgroup%
}}\!\right]$}%
}%
\ifdim\wd\matricesbox>\halfwidth\myboxwidth=\hsize\else\myboxwidth=\halfwidth\fi
\vbox{%
\ifdim\myboxwidth=\hsize
\setbox\onelinebox=\hbox{%
\vbox{\hbox{%
$\Pi_{8,49}$ spans $L_{141.3}$%
}\hbox{%
$22\infty222\infty2\rtimes C_{2}$%
}%
}%
\hfill\copy\matricesbox
}%
\ifdim\wd\onelinebox>\myboxwidth
\hbox to \myboxwidth{%
$\Pi_{8,49}$ spans $L_{141.3}$%
\hfil
$22\infty222\infty2\rtimes C_{2}$%
}%
\box\matricesbox
\else
\hbox to \myboxwidth{%
\unhbox\onelinebox
}%
\fi
\else
\hbox to \myboxwidth{%
$\Pi_{8,49}$ spans $L_{141.3}$%
\hfil}%
\hbox to \myboxwidth{%
$22\infty222\infty2\rtimes C_{2}$%
\hfil}%
\box\matricesbox
\fi
}%
\hfill\discretionary{}{}{}%
\setbox\matricesbox=\hbox{%
{$\left[\!\llap{\phantom{%
\begingroup \smaller\smaller\smaller\begin{tabular}{@{}c@{}}%
\phantom{0}\\\phantom{0}\\\phantom{0}
\end{tabular}\endgroup%
}}\right.$}%
\begingroup \smaller\smaller\smaller\begin{tabular}{@{}c@{}}%
-1/2\\\phantom{0}\\\phantom{0}
\end{tabular}\endgroup%
\kern3pt%
\begingroup \smaller\smaller\smaller\begin{tabular}{@{}c@{}}%
\phantom{0}\\3/2\\-1/2
\end{tabular}\endgroup%
\kern3pt%
\begingroup \smaller\smaller\smaller\begin{tabular}{@{}c@{}}%
\phantom{0}\\-1/2\\15/2
\end{tabular}\endgroup%
{$\left.\llap{\phantom{%
\begingroup \smaller\smaller\smaller\begin{tabular}{@{}c@{}}%
\phantom{0}\\\phantom{0}\\\phantom{0}
\end{tabular}\endgroup%
}}\!\right]$}%
{$\left[\!\llap{\phantom{%
\begingroup \smaller\smaller\smaller\begin{tabular}{@{}c@{}}%
0\\0\\0
\end{tabular}\endgroup%
}}\right.$}%
\begingroup \smaller\smaller\smaller\begin{tabular}{@{}c@{}}%
1\\1\\0
\end{tabular}\endgroup%
\kern3pt%
\begingroup \smaller\smaller\smaller\begin{tabular}{@{}c@{}}%
8\\2\\-2
\end{tabular}\endgroup%
\kern3pt%
\begingroup \smaller\smaller\smaller\begin{tabular}{@{}c@{}}%
22\\-2\\-6
\end{tabular}\endgroup%
\kern3pt%
\begingroup \smaller\smaller\smaller\begin{tabular}{@{}c@{}}%
22\\-9\\-5
\end{tabular}\endgroup%
{$\left.\llap{\phantom{%
\begingroup \smaller\smaller\smaller\begin{tabular}{@{}c@{}}%
0\\0\\0
\end{tabular}\endgroup%
}}\!\right]$}%
}%
\ifdim\wd\matricesbox>\halfwidth\myboxwidth=\hsize\else\myboxwidth=\halfwidth\fi
\vbox{%
\ifdim\myboxwidth=\hsize
\setbox\onelinebox=\hbox{%
\vbox{\hbox{%
$\Pi_{8,50}$ spans $L_{12.5}$%
}\hbox{%
$22\infty222\infty2\rtimes C_{2}$%
}%
}%
\hfill\copy\matricesbox
}%
\ifdim\wd\onelinebox>\myboxwidth
\hbox to \myboxwidth{%
$\Pi_{8,50}$ spans $L_{12.5}$%
\hfil
$22\infty222\infty2\rtimes C_{2}$%
}%
\box\matricesbox
\else
\hbox to \myboxwidth{%
\unhbox\onelinebox
}%
\fi
\else
\hbox to \myboxwidth{%
$\Pi_{8,50}$ spans $L_{12.5}$%
\hfil}%
\hbox to \myboxwidth{%
$22\infty222\infty2\rtimes C_{2}$%
\hfil}%
\box\matricesbox
\fi
}%
\hfill\discretionary{}{}{}%
\setbox\matricesbox=\hbox{%
{$\left[\!\llap{\phantom{%
\begingroup \smaller\smaller\smaller\begin{tabular}{@{}c@{}}%
\phantom{0}\\\phantom{0}\\\phantom{0}
\end{tabular}\endgroup%
}}\right.$}%
\begingroup \smaller\smaller\smaller\begin{tabular}{@{}c@{}}%
-1/2\\\phantom{0}\\\phantom{0}
\end{tabular}\endgroup%
\kern3pt%
\begingroup \smaller\smaller\smaller\begin{tabular}{@{}c@{}}%
\phantom{0}\\15/2\\-3/2
\end{tabular}\endgroup%
\kern3pt%
\begingroup \smaller\smaller\smaller\begin{tabular}{@{}c@{}}%
\phantom{0}\\-3/2\\15/2
\end{tabular}\endgroup%
{$\left.\llap{\phantom{%
\begingroup \smaller\smaller\smaller\begin{tabular}{@{}c@{}}%
\phantom{0}\\\phantom{0}\\\phantom{0}
\end{tabular}\endgroup%
}}\!\right]$}%
{$\left[\!\llap{\phantom{%
\begingroup \smaller\smaller\smaller\begin{tabular}{@{}c@{}}%
0\\0\\0
\end{tabular}\endgroup%
}}\right.$}%
\begingroup \smaller\smaller\smaller\begin{tabular}{@{}c@{}}%
4\\-1\\-1
\end{tabular}\endgroup%
\kern3pt%
\begingroup \smaller\smaller\smaller\begin{tabular}{@{}c@{}}%
18\\-1\\-5
\end{tabular}\endgroup%
\kern3pt%
\begingroup \smaller\smaller\smaller\begin{tabular}{@{}c@{}}%
12\\1\\-3
\end{tabular}\endgroup%
\kern3pt%
\begingroup \smaller\smaller\smaller\begin{tabular}{@{}c@{}}%
3\\1\\0
\end{tabular}\endgroup%
{$\left.\llap{\phantom{%
\begingroup \smaller\smaller\smaller\begin{tabular}{@{}c@{}}%
0\\0\\0
\end{tabular}\endgroup%
}}\!\right]$}%
}%
\ifdim\wd\matricesbox>\halfwidth\myboxwidth=\hsize\else\myboxwidth=\halfwidth\fi
\vbox{%
\ifdim\myboxwidth=\hsize
\setbox\onelinebox=\hbox{%
\vbox{\hbox{%
$\Pi_{8,51}$ spans $L_{4.12}$%
}\hbox{%
$22\infty222\infty2\rtimes C_{2}$%
}%
}%
\hfill\copy\matricesbox
}%
\ifdim\wd\onelinebox>\myboxwidth
\hbox to \myboxwidth{%
$\Pi_{8,51}$ spans $L_{4.12}$%
\hfil
$22\infty222\infty2\rtimes C_{2}$%
}%
\box\matricesbox
\else
\hbox to \myboxwidth{%
\unhbox\onelinebox
}%
\fi
\else
\hbox to \myboxwidth{%
$\Pi_{8,51}$ spans $L_{4.12}$%
\hfil}%
\hbox to \myboxwidth{%
$22\infty222\infty2\rtimes C_{2}$%
\hfil}%
\box\matricesbox
\fi
}%
\hfill\discretionary{}{}{}%
\setbox\matricesbox=\hbox{%
{$\left[\!\llap{\phantom{%
\begingroup \smaller\smaller\smaller\begin{tabular}{@{}c@{}}%
\phantom{0}\\\phantom{0}\\\phantom{0}
\end{tabular}\endgroup%
}}\right.$}%
\begingroup \smaller\smaller\smaller\begin{tabular}{@{}c@{}}%
-1/8\\\phantom{0}\\\phantom{0}
\end{tabular}\endgroup%
\kern3pt%
\begingroup \smaller\smaller\smaller\begin{tabular}{@{}c@{}}%
\phantom{0}\\45/2\\-5/2
\end{tabular}\endgroup%
\kern3pt%
\begingroup \smaller\smaller\smaller\begin{tabular}{@{}c@{}}%
\phantom{0}\\-5/2\\45/2
\end{tabular}\endgroup%
{$\left.\llap{\phantom{%
\begingroup \smaller\smaller\smaller\begin{tabular}{@{}c@{}}%
\phantom{0}\\\phantom{0}\\\phantom{0}
\end{tabular}\endgroup%
}}\!\right]$}%
{$\left[\!\llap{\phantom{%
\begingroup \smaller\smaller\smaller\begin{tabular}{@{}c@{}}%
0\\0\\0
\end{tabular}\endgroup%
}}\right.$}%
\begingroup \smaller\smaller\smaller\begin{tabular}{@{}c@{}}%
32\\-2\\-2
\end{tabular}\endgroup%
\kern3pt%
\begingroup \smaller\smaller\smaller\begin{tabular}{@{}c@{}}%
50\\-4\\-1
\end{tabular}\endgroup%
\kern3pt%
\begingroup \smaller\smaller\smaller\begin{tabular}{@{}c@{}}%
40\\-3\\1
\end{tabular}\endgroup%
\kern3pt%
\begingroup \smaller\smaller\smaller\begin{tabular}{@{}c@{}}%
10\\0\\1
\end{tabular}\endgroup%
{$\left.\llap{\phantom{%
\begingroup \smaller\smaller\smaller\begin{tabular}{@{}c@{}}%
0\\0\\0
\end{tabular}\endgroup%
}}\!\right]$}%
}%
\ifdim\wd\matricesbox>\halfwidth\myboxwidth=\hsize\else\myboxwidth=\halfwidth\fi
\vbox{%
\ifdim\myboxwidth=\hsize
\setbox\onelinebox=\hbox{%
\vbox{\hbox{%
$\Pi_{8,52}$ spans $L_{10.1}$%
}\hbox{%
$22\infty222\infty2\rtimes C_{2}$%
}%
}%
\hfill\copy\matricesbox
}%
\ifdim\wd\onelinebox>\myboxwidth
\hbox to \myboxwidth{%
$\Pi_{8,52}$ spans $L_{10.1}$%
\hfil
$22\infty222\infty2\rtimes C_{2}$%
}%
\box\matricesbox
\else
\hbox to \myboxwidth{%
\unhbox\onelinebox
}%
\fi
\else
\hbox to \myboxwidth{%
$\Pi_{8,52}$ spans $L_{10.1}$%
\hfil}%
\hbox to \myboxwidth{%
$22\infty222\infty2\rtimes C_{2}$%
\hfil}%
\box\matricesbox
\fi
}%
\hfill\discretionary{}{}{}%
\setbox\matricesbox=\hbox{%
{$\left[\!\llap{\phantom{%
\begingroup \smaller\smaller\smaller\begin{tabular}{@{}c@{}}%
\phantom{0}\\\phantom{0}\\\phantom{0}
\end{tabular}\endgroup%
}}\right.$}%
\begingroup \smaller\smaller\smaller\begin{tabular}{@{}c@{}}%
-1/8\\\phantom{0}\\\phantom{0}
\end{tabular}\endgroup%
\kern3pt%
\begingroup \smaller\smaller\smaller\begin{tabular}{@{}c@{}}%
\phantom{0}\\45/2\\-15/2
\end{tabular}\endgroup%
\kern3pt%
\begingroup \smaller\smaller\smaller\begin{tabular}{@{}c@{}}%
\phantom{0}\\-15/2\\45/2
\end{tabular}\endgroup%
{$\left.\llap{\phantom{%
\begingroup \smaller\smaller\smaller\begin{tabular}{@{}c@{}}%
\phantom{0}\\\phantom{0}\\\phantom{0}
\end{tabular}\endgroup%
}}\!\right]$}%
{$\left[\!\llap{\phantom{%
\begingroup \smaller\smaller\smaller\begin{tabular}{@{}c@{}}%
0\\0\\0
\end{tabular}\endgroup%
}}\right.$}%
\begingroup \smaller\smaller\smaller\begin{tabular}{@{}c@{}}%
12\\1\\1
\end{tabular}\endgroup%
\kern3pt%
\begingroup \smaller\smaller\smaller\begin{tabular}{@{}c@{}}%
60\\1\\5
\end{tabular}\endgroup%
\kern3pt%
\begingroup \smaller\smaller\smaller\begin{tabular}{@{}c@{}}%
30\\-1\\2
\end{tabular}\endgroup%
\kern3pt%
\begingroup \smaller\smaller\smaller\begin{tabular}{@{}c@{}}%
10\\-1\\0
\end{tabular}\endgroup%
{$\left.\llap{\phantom{%
\begingroup \smaller\smaller\smaller\begin{tabular}{@{}c@{}}%
0\\0\\0
\end{tabular}\endgroup%
}}\!\right]$}%
}%
\ifdim\wd\matricesbox>\halfwidth\myboxwidth=\hsize\else\myboxwidth=\halfwidth\fi
\vbox{%
\ifdim\myboxwidth=\hsize
\setbox\onelinebox=\hbox{%
\vbox{\hbox{%
$\Pi_{8,53}$ spans $L_{16.16}$%
}\hbox{%
$22222222\rtimes C_{2}$%
}%
}%
\hfill\copy\matricesbox
}%
\ifdim\wd\onelinebox>\myboxwidth
\hbox to \myboxwidth{%
$\Pi_{8,53}$ spans $L_{16.16}$%
\hfil
$22222222\rtimes C_{2}$%
}%
\box\matricesbox
\else
\hbox to \myboxwidth{%
\unhbox\onelinebox
}%
\fi
\else
\hbox to \myboxwidth{%
$\Pi_{8,53}$ spans $L_{16.16}$%
\hfil}%
\hbox to \myboxwidth{%
$22222222\rtimes C_{2}$%
\hfil}%
\box\matricesbox
\fi
}%
\hfill\discretionary{}{}{}%
\setbox\matricesbox=\hbox{%
{$\left[\!\llap{\phantom{%
\begingroup \smaller\smaller\smaller\begin{tabular}{@{}c@{}}%
\phantom{0}\\\phantom{0}\\\phantom{0}
\end{tabular}\endgroup%
}}\right.$}%
\begingroup \smaller\smaller\smaller\begin{tabular}{@{}c@{}}%
-1\\\phantom{0}\\\phantom{0}
\end{tabular}\endgroup%
\kern3pt%
\begingroup \smaller\smaller\smaller\begin{tabular}{@{}c@{}}%
\phantom{0}\\2\\-1
\end{tabular}\endgroup%
\kern3pt%
\begingroup \smaller\smaller\smaller\begin{tabular}{@{}c@{}}%
\phantom{0}\\-1\\2
\end{tabular}\endgroup%
{$\left.\llap{\phantom{%
\begingroup \smaller\smaller\smaller\begin{tabular}{@{}c@{}}%
\phantom{0}\\\phantom{0}\\\phantom{0}
\end{tabular}\endgroup%
}}\!\right]$}%
{$\left[\!\llap{\phantom{%
\begingroup \smaller\smaller\smaller\begin{tabular}{@{}c@{}}%
0\\0\\0
\end{tabular}\endgroup%
}}\right.$}%
\begingroup \smaller\smaller\smaller\begin{tabular}{@{}c@{}}%
2\\1\\-1
\end{tabular}\endgroup%
\kern3pt%
\begingroup \smaller\smaller\smaller\begin{tabular}{@{}c@{}}%
6\\-1\\-5
\end{tabular}\endgroup%
\kern3pt%
\begingroup \smaller\smaller\smaller\begin{tabular}{@{}c@{}}%
1\\-1\\-1
\end{tabular}\endgroup%
\kern3pt%
\begingroup \smaller\smaller\smaller\begin{tabular}{@{}c@{}}%
6\\-5\\-1
\end{tabular}\endgroup%
\kern3pt%
\begingroup \smaller\smaller\smaller\begin{tabular}{@{}c@{}}%
2\\-1\\1
\end{tabular}\endgroup%
\kern3pt%
\begingroup \smaller\smaller\smaller\begin{tabular}{@{}c@{}}%
2\\1\\2
\end{tabular}\endgroup%
\kern3pt%
\begingroup \smaller\smaller\smaller\begin{tabular}{@{}c@{}}%
6\\5\\4
\end{tabular}\endgroup%
\kern3pt%
\begingroup \smaller\smaller\smaller\begin{tabular}{@{}c@{}}%
6\\5\\1
\end{tabular}\endgroup%
{$\left.\llap{\phantom{%
\begingroup \smaller\smaller\smaller\begin{tabular}{@{}c@{}}%
0\\0\\0
\end{tabular}\endgroup%
}}\!\right]$}%
}%
\ifdim\wd\matricesbox>\halfwidth\myboxwidth=\hsize\else\myboxwidth=\halfwidth\fi
\vbox{%
\ifdim\myboxwidth=\hsize
\setbox\onelinebox=\hbox{%
\vbox{\hbox{%
$\Pi_{8,54}$ spans $L_{3.2}$%
}\hbox{%
$22223232$%
}%
}%
\hfill\copy\matricesbox
}%
\ifdim\wd\onelinebox>\myboxwidth
\hbox to \myboxwidth{%
$\Pi_{8,54}$ spans $L_{3.2}$%
\hfil
$22223232$%
}%
\box\matricesbox
\else
\hbox to \myboxwidth{%
\unhbox\onelinebox
}%
\fi
\else
\hbox to \myboxwidth{%
$\Pi_{8,54}$ spans $L_{3.2}$%
\hfil}%
\hbox to \myboxwidth{%
$22223232$%
\hfil}%
\box\matricesbox
\fi
}%
\hfill\discretionary{}{}{}%
\setbox\matricesbox=\hbox{%
{$\left[\!\llap{\phantom{%
\begingroup \smaller\smaller\smaller\begin{tabular}{@{}c@{}}%
\phantom{0}\\\phantom{0}\\\phantom{0}
\end{tabular}\endgroup%
}}\right.$}%
\begingroup \smaller\smaller\smaller\begin{tabular}{@{}c@{}}%
-1\\\phantom{0}\\\phantom{0}
\end{tabular}\endgroup%
\kern3pt%
\begingroup \smaller\smaller\smaller\begin{tabular}{@{}c@{}}%
\phantom{0}\\6\\-3
\end{tabular}\endgroup%
\kern3pt%
\begingroup \smaller\smaller\smaller\begin{tabular}{@{}c@{}}%
\phantom{0}\\-3\\6
\end{tabular}\endgroup%
{$\left.\llap{\phantom{%
\begingroup \smaller\smaller\smaller\begin{tabular}{@{}c@{}}%
\phantom{0}\\\phantom{0}\\\phantom{0}
\end{tabular}\endgroup%
}}\!\right]$}%
{$\left[\!\llap{\phantom{%
\begingroup \smaller\smaller\smaller\begin{tabular}{@{}c@{}}%
0\\0\\0
\end{tabular}\endgroup%
}}\right.$}%
\begingroup \smaller\smaller\smaller\begin{tabular}{@{}c@{}}%
2\\1\\1
\end{tabular}\endgroup%
\kern3pt%
\begingroup \smaller\smaller\smaller\begin{tabular}{@{}c@{}}%
2\\0\\1
\end{tabular}\endgroup%
\kern3pt%
\begingroup \smaller\smaller\smaller\begin{tabular}{@{}c@{}}%
2\\-1\\0
\end{tabular}\endgroup%
\kern3pt%
\begingroup \smaller\smaller\smaller\begin{tabular}{@{}c@{}}%
2\\-1\\-1
\end{tabular}\endgroup%
\kern3pt%
\begingroup \smaller\smaller\smaller\begin{tabular}{@{}c@{}}%
6\\-1\\-3
\end{tabular}\endgroup%
\kern3pt%
\begingroup \smaller\smaller\smaller\begin{tabular}{@{}c@{}}%
6\\1\\-2
\end{tabular}\endgroup%
\kern3pt%
\begingroup \smaller\smaller\smaller\begin{tabular}{@{}c@{}}%
18\\7\\-1
\end{tabular}\endgroup%
\kern3pt%
\begingroup \smaller\smaller\smaller\begin{tabular}{@{}c@{}}%
6\\3\\1
\end{tabular}\endgroup%
{$\left.\llap{\phantom{%
\begingroup \smaller\smaller\smaller\begin{tabular}{@{}c@{}}%
0\\0\\0
\end{tabular}\endgroup%
}}\!\right]$}%
}%
\ifdim\wd\matricesbox>\halfwidth\myboxwidth=\hsize\else\myboxwidth=\halfwidth\fi
\vbox{%
\ifdim\myboxwidth=\hsize
\setbox\onelinebox=\hbox{%
\vbox{\hbox{%
$\Pi_{8,55}$ spans $L_{155.1}$%
}\hbox{%
$33323622$%
}%
}%
\hfill\copy\matricesbox
}%
\ifdim\wd\onelinebox>\myboxwidth
\hbox to \myboxwidth{%
$\Pi_{8,55}$ spans $L_{155.1}$%
\hfil
$33323622$%
}%
\box\matricesbox
\else
\hbox to \myboxwidth{%
\unhbox\onelinebox
}%
\fi
\else
\hbox to \myboxwidth{%
$\Pi_{8,55}$ spans $L_{155.1}$%
\hfil}%
\hbox to \myboxwidth{%
$33323622$%
\hfil}%
\box\matricesbox
\fi
}%
\hfill\discretionary{}{}{}%

\vskip2pt\hrule\vskip2pt

\leavevmode\setbox\matricesbox=\hbox{%
{$\left[\!\llap{\phantom{%
\begingroup \smaller\smaller\smaller\begin{tabular}{@{}c@{}}%
\phantom{0}\\\phantom{0}\\\phantom{0}\\\phantom{0}
\end{tabular}\endgroup%
}}\right.$}%
\begingroup \smaller\smaller\smaller\begin{tabular}{@{}c@{}}%
-3\\\phantom{0}\\\phantom{0}\\\phantom{0}
\end{tabular}\endgroup%
\kern3pt%
\begingroup \smaller\smaller\smaller\begin{tabular}{@{}c@{}}%
\phantom{0}\\2/3\\\phantom{0}\\\phantom{0}
\end{tabular}\endgroup%
\kern3pt%
\begingroup \smaller\smaller\smaller\begin{tabular}{@{}c@{}}%
\phantom{0}\\\phantom{0}\\2/3\\\phantom{0}
\end{tabular}\endgroup%
\kern3pt%
\begingroup \smaller\smaller\smaller\begin{tabular}{@{}c@{}}%
\phantom{0}\\\phantom{0}\\\phantom{0}\\2/3
\end{tabular}\endgroup%
{$\left.\llap{\phantom{%
\begingroup \smaller\smaller\smaller\begin{tabular}{@{}c@{}}%
\phantom{0}\\\phantom{0}\\\phantom{0}\\\phantom{0}
\end{tabular}\endgroup%
}}\!\right]$}%
{$\left[\!\llap{\phantom{%
\begingroup \smaller\smaller\smaller\begin{tabular}{@{}c@{}}%
0\\0\\0\\0
\end{tabular}\endgroup%
}}\right.$}%
\begingroup \smaller\smaller\smaller\begin{tabular}{@{}c@{}}%
4\\-7\\2\\5
\end{tabular}\endgroup%
\kern3pt%
\begingroup \smaller\smaller\smaller\begin{tabular}{@{}c@{}}%
1\\-1\\-1\\2
\end{tabular}\endgroup%
{$\left.\llap{\phantom{%
\begingroup \smaller\smaller\smaller\begin{tabular}{@{}c@{}}%
0\\0\\0\\0
\end{tabular}\endgroup%
}}\!\right]$}%
}%
\ifdim\wd\matricesbox>\halfwidth\myboxwidth=\hsize\else\myboxwidth=\halfwidth\fi
\vbox{%
\ifdim\myboxwidth=\hsize
\setbox\onelinebox=\hbox{%
\vbox{\hbox{%
$\Pi_{9,1}=\hbox{GN}_{53}$ spans $L_{221.3}$%
}\hbox{%
$\slashthree\infty|\infty\slashthree\infty|\infty\slashthree\infty|\infty\rtimes D_{6}$%
}%
}%
\hfill\copy\matricesbox
}%
\ifdim\wd\onelinebox>\myboxwidth
\hbox to \myboxwidth{%
$\Pi_{9,1}=\hbox{GN}_{53}$ spans $L_{221.3}$%
\hfil
$\slashthree\infty|\infty\slashthree\infty|\infty\slashthree\infty|\infty\rtimes D_{6}$%
}%
\box\matricesbox
\else
\hbox to \myboxwidth{%
\unhbox\onelinebox
}%
\fi
\else
\hbox to \myboxwidth{%
$\Pi_{9,1}=\hbox{GN}_{53}$ spans $L_{221.3}$%
\hfil}%
\hbox to \myboxwidth{%
$\slashthree\infty|\infty\slashthree\infty|\infty\slashthree\infty|\infty\rtimes D_{6}$%
\hfil}%
\box\matricesbox
\fi
}%
\hfill\discretionary{}{}{}%
\setbox\matricesbox=\hbox{%
{$\left[\!\llap{\phantom{%
\begingroup \smaller\smaller\smaller\begin{tabular}{@{}c@{}}%
\phantom{0}\\\phantom{0}\\\phantom{0}\\\phantom{0}
\end{tabular}\endgroup%
}}\right.$}%
\begingroup \smaller\smaller\smaller\begin{tabular}{@{}c@{}}%
-1\\\phantom{0}\\\phantom{0}\\\phantom{0}
\end{tabular}\endgroup%
\kern3pt%
\begingroup \smaller\smaller\smaller\begin{tabular}{@{}c@{}}%
\phantom{0}\\1\\\phantom{0}\\\phantom{0}
\end{tabular}\endgroup%
\kern3pt%
\begingroup \smaller\smaller\smaller\begin{tabular}{@{}c@{}}%
\phantom{0}\\\phantom{0}\\1\\\phantom{0}
\end{tabular}\endgroup%
\kern3pt%
\begingroup \smaller\smaller\smaller\begin{tabular}{@{}c@{}}%
\phantom{0}\\\phantom{0}\\\phantom{0}\\1
\end{tabular}\endgroup%
{$\left.\llap{\phantom{%
\begingroup \smaller\smaller\smaller\begin{tabular}{@{}c@{}}%
\phantom{0}\\\phantom{0}\\\phantom{0}\\\phantom{0}
\end{tabular}\endgroup%
}}\!\right]$}%
{$\left[\!\llap{\phantom{%
\begingroup \smaller\smaller\smaller\begin{tabular}{@{}c@{}}%
0\\0\\0\\0
\end{tabular}\endgroup%
}}\right.$}%
\begingroup \smaller\smaller\smaller\begin{tabular}{@{}c@{}}%
6\\-5\\1\\4
\end{tabular}\endgroup%
\kern3pt%
\begingroup \smaller\smaller\smaller\begin{tabular}{@{}c@{}}%
2\\-1\\-1\\2
\end{tabular}\endgroup%
{$\left.\llap{\phantom{%
\begingroup \smaller\smaller\smaller\begin{tabular}{@{}c@{}}%
0\\0\\0\\0
\end{tabular}\endgroup%
}}\!\right]$}%
}%
\ifdim\wd\matricesbox>\halfwidth\myboxwidth=\hsize\else\myboxwidth=\halfwidth\fi
\vbox{%
\ifdim\myboxwidth=\hsize
\setbox\onelinebox=\hbox{%
\vbox{\hbox{%
$\Pi_{9,2}$ spans $L_{155.1}$%
}\hbox{%
$\slashthree2|2\slashthree2|2\slashthree2|2\rtimes D_{6}$%
}%
}%
\hfill\copy\matricesbox
}%
\ifdim\wd\onelinebox>\myboxwidth
\hbox to \myboxwidth{%
$\Pi_{9,2}$ spans $L_{155.1}$%
\hfil
$\slashthree2|2\slashthree2|2\slashthree2|2\rtimes D_{6}$%
}%
\box\matricesbox
\else
\hbox to \myboxwidth{%
\unhbox\onelinebox
}%
\fi
\else
\hbox to \myboxwidth{%
$\Pi_{9,2}$ spans $L_{155.1}$%
\hfil}%
\hbox to \myboxwidth{%
$\slashthree2|2\slashthree2|2\slashthree2|2\rtimes D_{6}$%
\hfil}%
\box\matricesbox
\fi
}%
\hfill\discretionary{}{}{}%
\setbox\matricesbox=\hbox{%
{$\left[\!\llap{\phantom{%
\begingroup \smaller\smaller\smaller\begin{tabular}{@{}c@{}}%
\phantom{0}\\\phantom{0}\\\phantom{0}
\end{tabular}\endgroup%
}}\right.$}%
\begingroup \smaller\smaller\smaller\begin{tabular}{@{}c@{}}%
-1/4\\\phantom{0}\\\phantom{0}
\end{tabular}\endgroup%
\kern3pt%
\begingroup \smaller\smaller\smaller\begin{tabular}{@{}c@{}}%
\phantom{0}\\15\\\phantom{0}
\end{tabular}\endgroup%
\kern3pt%
\begingroup \smaller\smaller\smaller\begin{tabular}{@{}c@{}}%
\phantom{0}\\\phantom{0}\\15
\end{tabular}\endgroup%
{$\left.\llap{\phantom{%
\begingroup \smaller\smaller\smaller\begin{tabular}{@{}c@{}}%
\phantom{0}\\\phantom{0}\\\phantom{0}
\end{tabular}\endgroup%
}}\!\right]$}%
{$\left[\!\llap{\phantom{%
\begingroup \smaller\smaller\smaller\begin{tabular}{@{}c@{}}%
0\\0\\0
\end{tabular}\endgroup%
}}\right.$}%
\begingroup \smaller\smaller\smaller\begin{tabular}{@{}c@{}}%
30\\4\\-1
\end{tabular}\endgroup%
\kern3pt%
\begingroup \smaller\smaller\smaller\begin{tabular}{@{}c@{}}%
20\\2\\-2
\end{tabular}\endgroup%
\kern3pt%
\begingroup \smaller\smaller\smaller\begin{tabular}{@{}c@{}}%
6\\0\\-1
\end{tabular}\endgroup%
\kern3pt%
\begingroup \smaller\smaller\smaller\begin{tabular}{@{}c@{}}%
20\\-2\\-2
\end{tabular}\endgroup%
\kern3pt%
\begingroup \smaller\smaller\smaller\begin{tabular}{@{}c@{}}%
6\\-1\\0
\end{tabular}\endgroup%
{$\left.\llap{\phantom{%
\begingroup \smaller\smaller\smaller\begin{tabular}{@{}c@{}}%
0\\0\\0
\end{tabular}\endgroup%
}}\!\right]$}%
}%
\ifdim\wd\matricesbox>\halfwidth\myboxwidth=\hsize\else\myboxwidth=\halfwidth\fi
\vbox{%
\ifdim\myboxwidth=\hsize
\setbox\onelinebox=\hbox{%
\vbox{\hbox{%
$\Pi_{9,3}$ spans $L_{19.10}$%
}\hbox{%
$22\slashtwo2222|22\rtimes D_{2}$%
}%
}%
\hfill\copy\matricesbox
}%
\ifdim\wd\onelinebox>\myboxwidth
\hbox to \myboxwidth{%
$\Pi_{9,3}$ spans $L_{19.10}$%
\hfil
$22\slashtwo2222|22\rtimes D_{2}$%
}%
\box\matricesbox
\else
\hbox to \myboxwidth{%
\unhbox\onelinebox
}%
\fi
\else
\hbox to \myboxwidth{%
$\Pi_{9,3}$ spans $L_{19.10}$%
\hfil}%
\hbox to \myboxwidth{%
$22\slashtwo2222|22\rtimes D_{2}$%
\hfil}%
\box\matricesbox
\fi
}%
\hfill\discretionary{}{}{}%
\setbox\matricesbox=\hbox{%
{$\left[\!\llap{\phantom{%
\begingroup \smaller\smaller\smaller\begin{tabular}{@{}c@{}}%
\phantom{0}\\\phantom{0}\\\phantom{0}
\end{tabular}\endgroup%
}}\right.$}%
\begingroup \smaller\smaller\smaller\begin{tabular}{@{}c@{}}%
-1\\\phantom{0}\\\phantom{0}
\end{tabular}\endgroup%
\kern3pt%
\begingroup \smaller\smaller\smaller\begin{tabular}{@{}c@{}}%
\phantom{0}\\6\\\phantom{0}
\end{tabular}\endgroup%
\kern3pt%
\begingroup \smaller\smaller\smaller\begin{tabular}{@{}c@{}}%
\phantom{0}\\\phantom{0}\\6
\end{tabular}\endgroup%
{$\left.\llap{\phantom{%
\begingroup \smaller\smaller\smaller\begin{tabular}{@{}c@{}}%
\phantom{0}\\\phantom{0}\\\phantom{0}
\end{tabular}\endgroup%
}}\!\right]$}%
{$\left[\!\llap{\phantom{%
\begingroup \smaller\smaller\smaller\begin{tabular}{@{}c@{}}%
0\\0\\0
\end{tabular}\endgroup%
}}\right.$}%
\begingroup \smaller\smaller\smaller\begin{tabular}{@{}c@{}}%
2\\1\\0
\end{tabular}\endgroup%
\kern3pt%
\begingroup \smaller\smaller\smaller\begin{tabular}{@{}c@{}}%
3\\1\\1
\end{tabular}\endgroup%
\kern3pt%
\begingroup \smaller\smaller\smaller\begin{tabular}{@{}c@{}}%
2\\0\\1
\end{tabular}\endgroup%
\kern3pt%
\begingroup \smaller\smaller\smaller\begin{tabular}{@{}c@{}}%
3\\-1\\1
\end{tabular}\endgroup%
\kern3pt%
\begingroup \smaller\smaller\smaller\begin{tabular}{@{}c@{}}%
12\\-5\\1
\end{tabular}\endgroup%
{$\left.\llap{\phantom{%
\begingroup \smaller\smaller\smaller\begin{tabular}{@{}c@{}}%
0\\0\\0
\end{tabular}\endgroup%
}}\!\right]$}%
}%
\ifdim\wd\matricesbox>\halfwidth\myboxwidth=\hsize\else\myboxwidth=\halfwidth\fi
\vbox{%
\ifdim\myboxwidth=\hsize
\setbox\onelinebox=\hbox{%
\vbox{\hbox{%
$\Pi_{9,4}$ spans $L_{123.8}$%
}\hbox{%
$22|2222\slashtwo22\rtimes D_{2}$%
}%
}%
\hfill\copy\matricesbox
}%
\ifdim\wd\onelinebox>\myboxwidth
\hbox to \myboxwidth{%
$\Pi_{9,4}$ spans $L_{123.8}$%
\hfil
$22|2222\slashtwo22\rtimes D_{2}$%
}%
\box\matricesbox
\else
\hbox to \myboxwidth{%
\unhbox\onelinebox
}%
\fi
\else
\hbox to \myboxwidth{%
$\Pi_{9,4}$ spans $L_{123.8}$%
\hfil}%
\hbox to \myboxwidth{%
$22|2222\slashtwo22\rtimes D_{2}$%
\hfil}%
\box\matricesbox
\fi
}%
\hfill\discretionary{}{}{}%
\setbox\matricesbox=\hbox{%
{$\left[\!\llap{\phantom{%
\begingroup \smaller\smaller\smaller\begin{tabular}{@{}c@{}}%
\phantom{0}\\\phantom{0}\\\phantom{0}
\end{tabular}\endgroup%
}}\right.$}%
\begingroup \smaller\smaller\smaller\begin{tabular}{@{}c@{}}%
-1\\\phantom{0}\\\phantom{0}
\end{tabular}\endgroup%
\kern3pt%
\begingroup \smaller\smaller\smaller\begin{tabular}{@{}c@{}}%
\phantom{0}\\3\\\phantom{0}
\end{tabular}\endgroup%
\kern3pt%
\begingroup \smaller\smaller\smaller\begin{tabular}{@{}c@{}}%
\phantom{0}\\\phantom{0}\\3
\end{tabular}\endgroup%
{$\left.\llap{\phantom{%
\begingroup \smaller\smaller\smaller\begin{tabular}{@{}c@{}}%
\phantom{0}\\\phantom{0}\\\phantom{0}
\end{tabular}\endgroup%
}}\!\right]$}%
{$\left[\!\llap{\phantom{%
\begingroup \smaller\smaller\smaller\begin{tabular}{@{}c@{}}%
0\\0\\0
\end{tabular}\endgroup%
}}\right.$}%
\begingroup \smaller\smaller\smaller\begin{tabular}{@{}c@{}}%
3\\2\\0
\end{tabular}\endgroup%
\kern3pt%
\begingroup \smaller\smaller\smaller\begin{tabular}{@{}c@{}}%
2\\1\\1
\end{tabular}\endgroup%
\kern3pt%
\begingroup \smaller\smaller\smaller\begin{tabular}{@{}c@{}}%
3\\0\\2
\end{tabular}\endgroup%
\kern3pt%
\begingroup \smaller\smaller\smaller\begin{tabular}{@{}c@{}}%
2\\-1\\1
\end{tabular}\endgroup%
\kern3pt%
\begingroup \smaller\smaller\smaller\begin{tabular}{@{}c@{}}%
24\\-14\\2
\end{tabular}\endgroup%
{$\left.\llap{\phantom{%
\begingroup \smaller\smaller\smaller\begin{tabular}{@{}c@{}}%
0\\0\\0
\end{tabular}\endgroup%
}}\!\right]$}%
}%
\ifdim\wd\matricesbox>\halfwidth\myboxwidth=\hsize\else\myboxwidth=\halfwidth\fi
\vbox{%
\ifdim\myboxwidth=\hsize
\setbox\onelinebox=\hbox{%
\vbox{\hbox{%
$\Pi_{9,5}$ spans $L_{123.8}$%
}\hbox{%
$222|2222\slashtwo2\rtimes D_{2}$%
}%
}%
\hfill\copy\matricesbox
}%
\ifdim\wd\onelinebox>\myboxwidth
\hbox to \myboxwidth{%
$\Pi_{9,5}$ spans $L_{123.8}$%
\hfil
$222|2222\slashtwo2\rtimes D_{2}$%
}%
\box\matricesbox
\else
\hbox to \myboxwidth{%
\unhbox\onelinebox
}%
\fi
\else
\hbox to \myboxwidth{%
$\Pi_{9,5}$ spans $L_{123.8}$%
\hfil}%
\hbox to \myboxwidth{%
$222|2222\slashtwo2\rtimes D_{2}$%
\hfil}%
\box\matricesbox
\fi
}%
\hfill\discretionary{}{}{}%
\setbox\matricesbox=\hbox{%
{$\left[\!\llap{\phantom{%
\begingroup \smaller\smaller\smaller\begin{tabular}{@{}c@{}}%
\phantom{0}\\\phantom{0}\\\phantom{0}
\end{tabular}\endgroup%
}}\right.$}%
\begingroup \smaller\smaller\smaller\begin{tabular}{@{}c@{}}%
-1\\\phantom{0}\\\phantom{0}
\end{tabular}\endgroup%
\kern3pt%
\begingroup \smaller\smaller\smaller\begin{tabular}{@{}c@{}}%
\phantom{0}\\3/2\\\phantom{0}
\end{tabular}\endgroup%
\kern3pt%
\begingroup \smaller\smaller\smaller\begin{tabular}{@{}c@{}}%
\phantom{0}\\\phantom{0}\\9/2
\end{tabular}\endgroup%
{$\left.\llap{\phantom{%
\begingroup \smaller\smaller\smaller\begin{tabular}{@{}c@{}}%
\phantom{0}\\\phantom{0}\\\phantom{0}
\end{tabular}\endgroup%
}}\!\right]$}%
{$\left[\!\llap{\phantom{%
\begingroup \smaller\smaller\smaller\begin{tabular}{@{}c@{}}%
0\\0\\0
\end{tabular}\endgroup%
}}\right.$}%
\begingroup \smaller\smaller\smaller\begin{tabular}{@{}c@{}}%
2\\-2\\0
\end{tabular}\endgroup%
\kern3pt%
\begingroup \smaller\smaller\smaller\begin{tabular}{@{}c@{}}%
2\\-1\\-1
\end{tabular}\endgroup%
\kern3pt%
\begingroup \smaller\smaller\smaller\begin{tabular}{@{}c@{}}%
6\\1\\-3
\end{tabular}\endgroup%
\kern3pt%
\begingroup \smaller\smaller\smaller\begin{tabular}{@{}c@{}}%
18\\9\\-7
\end{tabular}\endgroup%
\kern3pt%
\begingroup \smaller\smaller\smaller\begin{tabular}{@{}c@{}}%
6\\5\\-1
\end{tabular}\endgroup%
{$\left.\llap{\phantom{%
\begingroup \smaller\smaller\smaller\begin{tabular}{@{}c@{}}%
0\\0\\0
\end{tabular}\endgroup%
}}\!\right]$}%
}%
\ifdim\wd\matricesbox>\halfwidth\myboxwidth=\hsize\else\myboxwidth=\halfwidth\fi
\vbox{%
\ifdim\myboxwidth=\hsize
\setbox\onelinebox=\hbox{%
\vbox{\hbox{%
$\Pi_{9,6}$ spans $L_{155.1}$%
}\hbox{%
$3|3226\slashthree622\rtimes D_{2}$%
}%
}%
\hfill\copy\matricesbox
}%
\ifdim\wd\onelinebox>\myboxwidth
\hbox to \myboxwidth{%
$\Pi_{9,6}$ spans $L_{155.1}$%
\hfil
$3|3226\slashthree622\rtimes D_{2}$%
}%
\box\matricesbox
\else
\hbox to \myboxwidth{%
\unhbox\onelinebox
}%
\fi
\else
\hbox to \myboxwidth{%
$\Pi_{9,6}$ spans $L_{155.1}$%
\hfil}%
\hbox to \myboxwidth{%
$3|3226\slashthree622\rtimes D_{2}$%
\hfil}%
\box\matricesbox
\fi
}%
\hfill\discretionary{}{}{}%
\setbox\matricesbox=\hbox{%
{$\left[\!\llap{\phantom{%
\begingroup \smaller\smaller\smaller\begin{tabular}{@{}c@{}}%
\phantom{0}\\\phantom{0}\\\phantom{0}
\end{tabular}\endgroup%
}}\right.$}%
\begingroup \smaller\smaller\smaller\begin{tabular}{@{}c@{}}%
-3\\\phantom{0}\\\phantom{0}
\end{tabular}\endgroup%
\kern3pt%
\begingroup \smaller\smaller\smaller\begin{tabular}{@{}c@{}}%
\phantom{0}\\1\\\phantom{0}
\end{tabular}\endgroup%
\kern3pt%
\begingroup \smaller\smaller\smaller\begin{tabular}{@{}c@{}}%
\phantom{0}\\\phantom{0}\\3
\end{tabular}\endgroup%
{$\left.\llap{\phantom{%
\begingroup \smaller\smaller\smaller\begin{tabular}{@{}c@{}}%
\phantom{0}\\\phantom{0}\\\phantom{0}
\end{tabular}\endgroup%
}}\!\right]$}%
{$\left[\!\llap{\phantom{%
\begingroup \smaller\smaller\smaller\begin{tabular}{@{}c@{}}%
0\\0\\0
\end{tabular}\endgroup%
}}\right.$}%
\begingroup \smaller\smaller\smaller\begin{tabular}{@{}c@{}}%
1\\-2\\0
\end{tabular}\endgroup%
\kern3pt%
\begingroup \smaller\smaller\smaller\begin{tabular}{@{}c@{}}%
1\\-1\\1
\end{tabular}\endgroup%
\kern3pt%
\begingroup \smaller\smaller\smaller\begin{tabular}{@{}c@{}}%
4\\2\\4
\end{tabular}\endgroup%
\kern3pt%
\begingroup \smaller\smaller\smaller\begin{tabular}{@{}c@{}}%
4\\5\\3
\end{tabular}\endgroup%
\kern3pt%
\begingroup \smaller\smaller\smaller\begin{tabular}{@{}c@{}}%
4\\7\\1
\end{tabular}\endgroup%
{$\left.\llap{\phantom{%
\begingroup \smaller\smaller\smaller\begin{tabular}{@{}c@{}}%
0\\0\\0
\end{tabular}\endgroup%
}}\!\right]$}%
}%
\ifdim\wd\matricesbox>\halfwidth\myboxwidth=\hsize\else\myboxwidth=\halfwidth\fi
\vbox{%
\ifdim\myboxwidth=\hsize
\setbox\onelinebox=\hbox{%
\vbox{\hbox{%
$\Pi_{9,7}=\hbox{GN}_{51}$ spans $L_{7.7}$%
}\hbox{%
$3\infty\infty|\infty\infty3\infty\slashthree\infty\rtimes D_{2}$%
}%
}%
\hfill\copy\matricesbox
}%
\ifdim\wd\onelinebox>\myboxwidth
\hbox to \myboxwidth{%
$\Pi_{9,7}=\hbox{GN}_{51}$ spans $L_{7.7}$%
\hfil
$3\infty\infty|\infty\infty3\infty\slashthree\infty\rtimes D_{2}$%
}%
\box\matricesbox
\else
\hbox to \myboxwidth{%
\unhbox\onelinebox
}%
\fi
\else
\hbox to \myboxwidth{%
$\Pi_{9,7}=\hbox{GN}_{51}$ spans $L_{7.7}$%
\hfil}%
\hbox to \myboxwidth{%
$3\infty\infty|\infty\infty3\infty\slashthree\infty\rtimes D_{2}$%
\hfil}%
\box\matricesbox
\fi
}%
\hfill\discretionary{}{}{}%
\setbox\matricesbox=\hbox{%
{$\left[\!\llap{\phantom{%
\begingroup \smaller\smaller\smaller\begin{tabular}{@{}c@{}}%
\phantom{0}\\\phantom{0}\\\phantom{0}
\end{tabular}\endgroup%
}}\right.$}%
\begingroup \smaller\smaller\smaller\begin{tabular}{@{}c@{}}%
-4\\\phantom{0}\\\phantom{0}
\end{tabular}\endgroup%
\kern3pt%
\begingroup \smaller\smaller\smaller\begin{tabular}{@{}c@{}}%
\phantom{0}\\1/2\\\phantom{0}
\end{tabular}\endgroup%
\kern3pt%
\begingroup \smaller\smaller\smaller\begin{tabular}{@{}c@{}}%
\phantom{0}\\\phantom{0}\\1/2
\end{tabular}\endgroup%
{$\left.\llap{\phantom{%
\begingroup \smaller\smaller\smaller\begin{tabular}{@{}c@{}}%
\phantom{0}\\\phantom{0}\\\phantom{0}
\end{tabular}\endgroup%
}}\!\right]$}%
{$\left[\!\llap{\phantom{%
\begingroup \smaller\smaller\smaller\begin{tabular}{@{}c@{}}%
0\\0\\0
\end{tabular}\endgroup%
}}\right.$}%
\begingroup \smaller\smaller\smaller\begin{tabular}{@{}c@{}}%
2\\6\\0
\end{tabular}\endgroup%
\kern3pt%
\begingroup \smaller\smaller\smaller\begin{tabular}{@{}c@{}}%
4\\10\\-6
\end{tabular}\endgroup%
\kern3pt%
\begingroup \smaller\smaller\smaller\begin{tabular}{@{}c@{}}%
1\\1\\-3
\end{tabular}\endgroup%
\kern3pt%
\begingroup \smaller\smaller\smaller\begin{tabular}{@{}c@{}}%
1\\-1\\-3
\end{tabular}\endgroup%
\kern3pt%
\begingroup \smaller\smaller\smaller\begin{tabular}{@{}c@{}}%
1\\-3\\-1
\end{tabular}\endgroup%
{$\left.\llap{\phantom{%
\begingroup \smaller\smaller\smaller\begin{tabular}{@{}c@{}}%
0\\0\\0
\end{tabular}\endgroup%
}}\!\right]$}%
}%
\ifdim\wd\matricesbox>\halfwidth\myboxwidth=\hsize\else\myboxwidth=\halfwidth\fi
\vbox{%
\ifdim\myboxwidth=\hsize
\setbox\onelinebox=\hbox{%
\vbox{\hbox{%
$\Pi_{9,8}$ spans $L_{145.1}$%
}\hbox{%
$|4\infty2\infty\slashtwo\infty2\infty4\rtimes D_{2}$%
}%
}%
\hfill\copy\matricesbox
}%
\ifdim\wd\onelinebox>\myboxwidth
\hbox to \myboxwidth{%
$\Pi_{9,8}$ spans $L_{145.1}$%
\hfil
$|4\infty2\infty\slashtwo\infty2\infty4\rtimes D_{2}$%
}%
\box\matricesbox
\else
\hbox to \myboxwidth{%
\unhbox\onelinebox
}%
\fi
\else
\hbox to \myboxwidth{%
$\Pi_{9,8}$ spans $L_{145.1}$%
\hfil}%
\hbox to \myboxwidth{%
$|4\infty2\infty\slashtwo\infty2\infty4\rtimes D_{2}$%
\hfil}%
\box\matricesbox
\fi
}%
\hfill\discretionary{}{}{}%
\setbox\matricesbox=\hbox{%
{$\left[\!\llap{\phantom{%
\begingroup \smaller\smaller\smaller\begin{tabular}{@{}c@{}}%
\phantom{0}\\\phantom{0}\\\phantom{0}
\end{tabular}\endgroup%
}}\right.$}%
\begingroup \smaller\smaller\smaller\begin{tabular}{@{}c@{}}%
-1\\\phantom{0}\\\phantom{0}
\end{tabular}\endgroup%
\kern3pt%
\begingroup \smaller\smaller\smaller\begin{tabular}{@{}c@{}}%
\phantom{0}\\5\\\phantom{0}
\end{tabular}\endgroup%
\kern3pt%
\begingroup \smaller\smaller\smaller\begin{tabular}{@{}c@{}}%
\phantom{0}\\\phantom{0}\\1
\end{tabular}\endgroup%
{$\left.\llap{\phantom{%
\begingroup \smaller\smaller\smaller\begin{tabular}{@{}c@{}}%
\phantom{0}\\\phantom{0}\\\phantom{0}
\end{tabular}\endgroup%
}}\!\right]$}%
{$\left[\!\llap{\phantom{%
\begingroup \smaller\smaller\smaller\begin{tabular}{@{}c@{}}%
0\\0\\0
\end{tabular}\endgroup%
}}\right.$}%
\begingroup \smaller\smaller\smaller\begin{tabular}{@{}c@{}}%
4\\-2\\0
\end{tabular}\endgroup%
\kern3pt%
\begingroup \smaller\smaller\smaller\begin{tabular}{@{}c@{}}%
20\\-8\\-10
\end{tabular}\endgroup%
\kern3pt%
\begingroup \smaller\smaller\smaller\begin{tabular}{@{}c@{}}%
5\\-1\\-5
\end{tabular}\endgroup%
\kern3pt%
\begingroup \smaller\smaller\smaller\begin{tabular}{@{}c@{}}%
5\\1\\-5
\end{tabular}\endgroup%
\kern3pt%
\begingroup \smaller\smaller\smaller\begin{tabular}{@{}c@{}}%
2\\1\\-1
\end{tabular}\endgroup%
{$\left.\llap{\phantom{%
\begingroup \smaller\smaller\smaller\begin{tabular}{@{}c@{}}%
0\\0\\0
\end{tabular}\endgroup%
}}\!\right]$}%
}%
\ifdim\wd\matricesbox>\halfwidth\myboxwidth=\hsize\else\myboxwidth=\halfwidth\fi
\vbox{%
\ifdim\myboxwidth=\hsize
\setbox\onelinebox=\hbox{%
\vbox{\hbox{%
$\Pi_{9,9}$ spans $L_{5.7}$%
}\hbox{%
$\infty2|2\infty\infty2\slashtwo2\infty\rtimes D_{2}$%
}%
}%
\hfill\copy\matricesbox
}%
\ifdim\wd\onelinebox>\myboxwidth
\hbox to \myboxwidth{%
$\Pi_{9,9}$ spans $L_{5.7}$%
\hfil
$\infty2|2\infty\infty2\slashtwo2\infty\rtimes D_{2}$%
}%
\box\matricesbox
\else
\hbox to \myboxwidth{%
\unhbox\onelinebox
}%
\fi
\else
\hbox to \myboxwidth{%
$\Pi_{9,9}$ spans $L_{5.7}$%
\hfil}%
\hbox to \myboxwidth{%
$\infty2|2\infty\infty2\slashtwo2\infty\rtimes D_{2}$%
\hfil}%
\box\matricesbox
\fi
}%
\hfill\discretionary{}{}{}%
\setbox\matricesbox=\hbox{%
{$\left[\!\llap{\phantom{%
\begingroup \smaller\smaller\smaller\begin{tabular}{@{}c@{}}%
\phantom{0}\\\phantom{0}\\\phantom{0}
\end{tabular}\endgroup%
}}\right.$}%
\begingroup \smaller\smaller\smaller\begin{tabular}{@{}c@{}}%
-1/8\\\phantom{0}\\\phantom{0}
\end{tabular}\endgroup%
\kern3pt%
\begingroup \smaller\smaller\smaller\begin{tabular}{@{}c@{}}%
\phantom{0}\\77/2\\-7/2
\end{tabular}\endgroup%
\kern3pt%
\begingroup \smaller\smaller\smaller\begin{tabular}{@{}c@{}}%
\phantom{0}\\-7/2\\77/2
\end{tabular}\endgroup%
{$\left.\llap{\phantom{%
\begingroup \smaller\smaller\smaller\begin{tabular}{@{}c@{}}%
\phantom{0}\\\phantom{0}\\\phantom{0}
\end{tabular}\endgroup%
}}\!\right]$}%
{$\left[\!\llap{\phantom{%
\begingroup \smaller\smaller\smaller\begin{tabular}{@{}c@{}}%
0\\0\\0
\end{tabular}\endgroup%
}}\right.$}%
\begingroup \smaller\smaller\smaller\begin{tabular}{@{}c@{}}%
20\\-1\\-1
\end{tabular}\endgroup%
\kern3pt%
\begingroup \smaller\smaller\smaller\begin{tabular}{@{}c@{}}%
84\\-1\\-5
\end{tabular}\endgroup%
\kern3pt%
\begingroup \smaller\smaller\smaller\begin{tabular}{@{}c@{}}%
70\\1\\-4
\end{tabular}\endgroup%
\kern3pt%
\begingroup \smaller\smaller\smaller\begin{tabular}{@{}c@{}}%
48\\2\\-2
\end{tabular}\endgroup%
\kern3pt%
\begingroup \smaller\smaller\smaller\begin{tabular}{@{}c@{}}%
14\\1\\0
\end{tabular}\endgroup%
\kern3pt%
\begingroup \smaller\smaller\smaller\begin{tabular}{@{}c@{}}%
20\\1\\1
\end{tabular}\endgroup%
\kern3pt%
\begingroup \smaller\smaller\smaller\begin{tabular}{@{}c@{}}%
14\\0\\1
\end{tabular}\endgroup%
\kern3pt%
\begingroup \smaller\smaller\smaller\begin{tabular}{@{}c@{}}%
48\\-2\\2
\end{tabular}\endgroup%
\kern3pt%
\begingroup \smaller\smaller\smaller\begin{tabular}{@{}c@{}}%
14\\-1\\0
\end{tabular}\endgroup%
{$\left.\llap{\phantom{%
\begingroup \smaller\smaller\smaller\begin{tabular}{@{}c@{}}%
0\\0\\0
\end{tabular}\endgroup%
}}\!\right]$}%
}%
\ifdim\wd\matricesbox>\halfwidth\myboxwidth=\hsize\else\myboxwidth=\halfwidth\fi
\vbox{%
\ifdim\myboxwidth=\hsize
\setbox\onelinebox=\hbox{%
\vbox{\hbox{%
$\Pi_{9,10}$ spans $L_{69.11}$%
}\hbox{%
$222222222$%
}%
}%
\hfill\copy\matricesbox
}%
\ifdim\wd\onelinebox>\myboxwidth
\hbox to \myboxwidth{%
$\Pi_{9,10}$ spans $L_{69.11}$%
\hfil
$222222222$%
}%
\box\matricesbox
\else
\hbox to \myboxwidth{%
\unhbox\onelinebox
}%
\fi
\else
\hbox to \myboxwidth{%
$\Pi_{9,10}$ spans $L_{69.11}$%
\hfil}%
\hbox to \myboxwidth{%
$222222222$%
\hfil}%
\box\matricesbox
\fi
}%
\hfill\discretionary{}{}{}%
\setbox\matricesbox=\hbox{%
{$\left[\!\llap{\phantom{%
\begingroup \smaller\smaller\smaller\begin{tabular}{@{}c@{}}%
\phantom{0}\\\phantom{0}\\\phantom{0}
\end{tabular}\endgroup%
}}\right.$}%
\begingroup \smaller\smaller\smaller\begin{tabular}{@{}c@{}}%
-1/2\\\phantom{0}\\\phantom{0}
\end{tabular}\endgroup%
\kern3pt%
\begingroup \smaller\smaller\smaller\begin{tabular}{@{}c@{}}%
\phantom{0}\\15/2\\\phantom{0}
\end{tabular}\endgroup%
\kern3pt%
\begingroup \smaller\smaller\smaller\begin{tabular}{@{}c@{}}%
\phantom{0}\\\phantom{0}\\10
\end{tabular}\endgroup%
{$\left.\llap{\phantom{%
\begingroup \smaller\smaller\smaller\begin{tabular}{@{}c@{}}%
\phantom{0}\\\phantom{0}\\\phantom{0}
\end{tabular}\endgroup%
}}\!\right]$}%
{$\left[\!\llap{\phantom{%
\begingroup \smaller\smaller\smaller\begin{tabular}{@{}c@{}}%
0\\0\\0
\end{tabular}\endgroup%
}}\right.$}%
\begingroup \smaller\smaller\smaller\begin{tabular}{@{}c@{}}%
3\\-1\\0
\end{tabular}\endgroup%
\kern3pt%
\begingroup \smaller\smaller\smaller\begin{tabular}{@{}c@{}}%
5\\-1\\-1
\end{tabular}\endgroup%
\kern3pt%
\begingroup \smaller\smaller\smaller\begin{tabular}{@{}c@{}}%
8\\0\\-2
\end{tabular}\endgroup%
\kern3pt%
\begingroup \smaller\smaller\smaller\begin{tabular}{@{}c@{}}%
5\\1\\-1
\end{tabular}\endgroup%
\kern3pt%
\begingroup \smaller\smaller\smaller\begin{tabular}{@{}c@{}}%
3\\1\\0
\end{tabular}\endgroup%
\kern3pt%
\begingroup \smaller\smaller\smaller\begin{tabular}{@{}c@{}}%
5\\1\\1
\end{tabular}\endgroup%
\kern3pt%
\begingroup \smaller\smaller\smaller\begin{tabular}{@{}c@{}}%
8\\0\\2
\end{tabular}\endgroup%
\kern3pt%
\begingroup \smaller\smaller\smaller\begin{tabular}{@{}c@{}}%
30\\-4\\6
\end{tabular}\endgroup%
\kern3pt%
\begingroup \smaller\smaller\smaller\begin{tabular}{@{}c@{}}%
40\\-8\\6
\end{tabular}\endgroup%
{$\left.\llap{\phantom{%
\begingroup \smaller\smaller\smaller\begin{tabular}{@{}c@{}}%
0\\0\\0
\end{tabular}\endgroup%
}}\!\right]$}%
}%
\ifdim\wd\matricesbox>\halfwidth\myboxwidth=\hsize\else\myboxwidth=\halfwidth\fi
\vbox{%
\ifdim\myboxwidth=\hsize
\setbox\onelinebox=\hbox{%
\vbox{\hbox{%
$\Pi_{9,11}$ spans $L_{17.11}$%
}\hbox{%
$222222222$%
}%
}%
\hfill\copy\matricesbox
}%
\ifdim\wd\onelinebox>\myboxwidth
\hbox to \myboxwidth{%
$\Pi_{9,11}$ spans $L_{17.11}$%
\hfil
$222222222$%
}%
\box\matricesbox
\else
\hbox to \myboxwidth{%
\unhbox\onelinebox
}%
\fi
\else
\hbox to \myboxwidth{%
$\Pi_{9,11}$ spans $L_{17.11}$%
\hfil}%
\hbox to \myboxwidth{%
$222222222$%
\hfil}%
\box\matricesbox
\fi
}%
\hfill\discretionary{}{}{}%
\setbox\matricesbox=\hbox{%
{$\left[\!\llap{\phantom{%
\begingroup \smaller\smaller\smaller\begin{tabular}{@{}c@{}}%
\phantom{0}\\\phantom{0}\\\phantom{0}
\end{tabular}\endgroup%
}}\right.$}%
\begingroup \smaller\smaller\smaller\begin{tabular}{@{}c@{}}%
-1/8\\\phantom{0}\\\phantom{0}
\end{tabular}\endgroup%
\kern3pt%
\begingroup \smaller\smaller\smaller\begin{tabular}{@{}c@{}}%
\phantom{0}\\45/2\\-5/2
\end{tabular}\endgroup%
\kern3pt%
\begingroup \smaller\smaller\smaller\begin{tabular}{@{}c@{}}%
\phantom{0}\\-5/2\\45/2
\end{tabular}\endgroup%
{$\left.\llap{\phantom{%
\begingroup \smaller\smaller\smaller\begin{tabular}{@{}c@{}}%
\phantom{0}\\\phantom{0}\\\phantom{0}
\end{tabular}\endgroup%
}}\!\right]$}%
{$\left[\!\llap{\phantom{%
\begingroup \smaller\smaller\smaller\begin{tabular}{@{}c@{}}%
0\\0\\0
\end{tabular}\endgroup%
}}\right.$}%
\begingroup \smaller\smaller\smaller\begin{tabular}{@{}c@{}}%
40\\1\\-3
\end{tabular}\endgroup%
\kern3pt%
\begingroup \smaller\smaller\smaller\begin{tabular}{@{}c@{}}%
50\\-1\\-4
\end{tabular}\endgroup%
\kern3pt%
\begingroup \smaller\smaller\smaller\begin{tabular}{@{}c@{}}%
32\\-2\\-2
\end{tabular}\endgroup%
\kern3pt%
\begingroup \smaller\smaller\smaller\begin{tabular}{@{}c@{}}%
50\\-4\\-1
\end{tabular}\endgroup%
\kern3pt%
\begingroup \smaller\smaller\smaller\begin{tabular}{@{}c@{}}%
40\\-3\\1
\end{tabular}\endgroup%
\kern3pt%
\begingroup \smaller\smaller\smaller\begin{tabular}{@{}c@{}}%
10\\0\\1
\end{tabular}\endgroup%
\kern3pt%
\begingroup \smaller\smaller\smaller\begin{tabular}{@{}c@{}}%
32\\2\\2
\end{tabular}\endgroup%
\kern3pt%
\begingroup \smaller\smaller\smaller\begin{tabular}{@{}c@{}}%
50\\4\\1
\end{tabular}\endgroup%
\kern3pt%
\begingroup \smaller\smaller\smaller\begin{tabular}{@{}c@{}}%
40\\3\\-1
\end{tabular}\endgroup%
{$\left.\llap{\phantom{%
\begingroup \smaller\smaller\smaller\begin{tabular}{@{}c@{}}%
0\\0\\0
\end{tabular}\endgroup%
}}\!\right]$}%
}%
\ifdim\wd\matricesbox>\halfwidth\myboxwidth=\hsize\else\myboxwidth=\halfwidth\fi
\vbox{%
\ifdim\myboxwidth=\hsize
\setbox\onelinebox=\hbox{%
\vbox{\hbox{%
$\Pi_{9,12}$ spans $L_{10.1}$%
}\hbox{%
$2222\infty222\infty$%
}%
}%
\hfill\copy\matricesbox
}%
\ifdim\wd\onelinebox>\myboxwidth
\hbox to \myboxwidth{%
$\Pi_{9,12}$ spans $L_{10.1}$%
\hfil
$2222\infty222\infty$%
}%
\box\matricesbox
\else
\hbox to \myboxwidth{%
\unhbox\onelinebox
}%
\fi
\else
\hbox to \myboxwidth{%
$\Pi_{9,12}$ spans $L_{10.1}$%
\hfil}%
\hbox to \myboxwidth{%
$2222\infty222\infty$%
\hfil}%
\box\matricesbox
\fi
}%
\hfill\discretionary{}{}{}%
\setbox\matricesbox=\hbox{%
{$\left[\!\llap{\phantom{%
\begingroup \smaller\smaller\smaller\begin{tabular}{@{}c@{}}%
\phantom{0}\\\phantom{0}\\\phantom{0}
\end{tabular}\endgroup%
}}\right.$}%
\begingroup \smaller\smaller\smaller\begin{tabular}{@{}c@{}}%
-1\\\phantom{0}\\\phantom{0}
\end{tabular}\endgroup%
\kern3pt%
\begingroup \smaller\smaller\smaller\begin{tabular}{@{}c@{}}%
\phantom{0}\\2\\\phantom{0}
\end{tabular}\endgroup%
\kern3pt%
\begingroup \smaller\smaller\smaller\begin{tabular}{@{}c@{}}%
\phantom{0}\\\phantom{0}\\4
\end{tabular}\endgroup%
{$\left.\llap{\phantom{%
\begingroup \smaller\smaller\smaller\begin{tabular}{@{}c@{}}%
\phantom{0}\\\phantom{0}\\\phantom{0}
\end{tabular}\endgroup%
}}\!\right]$}%
{$\left[\!\llap{\phantom{%
\begingroup \smaller\smaller\smaller\begin{tabular}{@{}c@{}}%
0\\0\\0
\end{tabular}\endgroup%
}}\right.$}%
\begingroup \smaller\smaller\smaller\begin{tabular}{@{}c@{}}%
1\\1\\0
\end{tabular}\endgroup%
\kern3pt%
\begingroup \smaller\smaller\smaller\begin{tabular}{@{}c@{}}%
16\\8\\-6
\end{tabular}\endgroup%
\kern3pt%
\begingroup \smaller\smaller\smaller\begin{tabular}{@{}c@{}}%
8\\2\\-4
\end{tabular}\endgroup%
\kern3pt%
\begingroup \smaller\smaller\smaller\begin{tabular}{@{}c@{}}%
2\\-1\\-1
\end{tabular}\endgroup%
\kern3pt%
\begingroup \smaller\smaller\smaller\begin{tabular}{@{}c@{}}%
1\\-1\\0
\end{tabular}\endgroup%
\kern3pt%
\begingroup \smaller\smaller\smaller\begin{tabular}{@{}c@{}}%
16\\-8\\6
\end{tabular}\endgroup%
\kern3pt%
\begingroup \smaller\smaller\smaller\begin{tabular}{@{}c@{}}%
8\\-2\\4
\end{tabular}\endgroup%
\kern3pt%
\begingroup \smaller\smaller\smaller\begin{tabular}{@{}c@{}}%
8\\2\\4
\end{tabular}\endgroup%
\kern3pt%
\begingroup \smaller\smaller\smaller\begin{tabular}{@{}c@{}}%
16\\8\\6
\end{tabular}\endgroup%
{$\left.\llap{\phantom{%
\begingroup \smaller\smaller\smaller\begin{tabular}{@{}c@{}}%
0\\0\\0
\end{tabular}\endgroup%
}}\!\right]$}%
}%
\ifdim\wd\matricesbox>\halfwidth\myboxwidth=\hsize\else\myboxwidth=\halfwidth\fi
\vbox{%
\ifdim\myboxwidth=\hsize
\setbox\onelinebox=\hbox{%
\vbox{\hbox{%
$\Pi_{9,13}$ spans $L_{141.3}$%
}\hbox{%
$22\infty222\infty22$%
}%
}%
\hfill\copy\matricesbox
}%
\ifdim\wd\onelinebox>\myboxwidth
\hbox to \myboxwidth{%
$\Pi_{9,13}$ spans $L_{141.3}$%
\hfil
$22\infty222\infty22$%
}%
\box\matricesbox
\else
\hbox to \myboxwidth{%
\unhbox\onelinebox
}%
\fi
\else
\hbox to \myboxwidth{%
$\Pi_{9,13}$ spans $L_{141.3}$%
\hfil}%
\hbox to \myboxwidth{%
$22\infty222\infty22$%
\hfil}%
\box\matricesbox
\fi
}%
\hfill\discretionary{}{}{}%
\setbox\matricesbox=\hbox{%
{$\left[\!\llap{\phantom{%
\begingroup \smaller\smaller\smaller\begin{tabular}{@{}c@{}}%
\phantom{0}\\\phantom{0}\\\phantom{0}
\end{tabular}\endgroup%
}}\right.$}%
\begingroup \smaller\smaller\smaller\begin{tabular}{@{}c@{}}%
-1/2\\\phantom{0}\\\phantom{0}
\end{tabular}\endgroup%
\kern3pt%
\begingroup \smaller\smaller\smaller\begin{tabular}{@{}c@{}}%
\phantom{0}\\15/2\\-3/2
\end{tabular}\endgroup%
\kern3pt%
\begingroup \smaller\smaller\smaller\begin{tabular}{@{}c@{}}%
\phantom{0}\\-3/2\\15/2
\end{tabular}\endgroup%
{$\left.\llap{\phantom{%
\begingroup \smaller\smaller\smaller\begin{tabular}{@{}c@{}}%
\phantom{0}\\\phantom{0}\\\phantom{0}
\end{tabular}\endgroup%
}}\!\right]$}%
{$\left[\!\llap{\phantom{%
\begingroup \smaller\smaller\smaller\begin{tabular}{@{}c@{}}%
0\\0\\0
\end{tabular}\endgroup%
}}\right.$}%
\begingroup \smaller\smaller\smaller\begin{tabular}{@{}c@{}}%
4\\-1\\-1
\end{tabular}\endgroup%
\kern3pt%
\begingroup \smaller\smaller\smaller\begin{tabular}{@{}c@{}}%
18\\-1\\-5
\end{tabular}\endgroup%
\kern3pt%
\begingroup \smaller\smaller\smaller\begin{tabular}{@{}c@{}}%
12\\1\\-3
\end{tabular}\endgroup%
\kern3pt%
\begingroup \smaller\smaller\smaller\begin{tabular}{@{}c@{}}%
3\\1\\0
\end{tabular}\endgroup%
\kern3pt%
\begingroup \smaller\smaller\smaller\begin{tabular}{@{}c@{}}%
4\\1\\1
\end{tabular}\endgroup%
\kern3pt%
\begingroup \smaller\smaller\smaller\begin{tabular}{@{}c@{}}%
18\\1\\5
\end{tabular}\endgroup%
\kern3pt%
\begingroup \smaller\smaller\smaller\begin{tabular}{@{}c@{}}%
12\\-1\\3
\end{tabular}\endgroup%
\kern3pt%
\begingroup \smaller\smaller\smaller\begin{tabular}{@{}c@{}}%
12\\-3\\1
\end{tabular}\endgroup%
\kern3pt%
\begingroup \smaller\smaller\smaller\begin{tabular}{@{}c@{}}%
18\\-5\\-1
\end{tabular}\endgroup%
{$\left.\llap{\phantom{%
\begingroup \smaller\smaller\smaller\begin{tabular}{@{}c@{}}%
0\\0\\0
\end{tabular}\endgroup%
}}\!\right]$}%
}%
\ifdim\wd\matricesbox>\halfwidth\myboxwidth=\hsize\else\myboxwidth=\halfwidth\fi
\vbox{%
\ifdim\myboxwidth=\hsize
\setbox\onelinebox=\hbox{%
\vbox{\hbox{%
$\Pi_{9,14}$ spans $L_{4.12}$%
}\hbox{%
$22\infty222\infty22$%
}%
}%
\hfill\copy\matricesbox
}%
\ifdim\wd\onelinebox>\myboxwidth
\hbox to \myboxwidth{%
$\Pi_{9,14}$ spans $L_{4.12}$%
\hfil
$22\infty222\infty22$%
}%
\box\matricesbox
\else
\hbox to \myboxwidth{%
\unhbox\onelinebox
}%
\fi
\else
\hbox to \myboxwidth{%
$\Pi_{9,14}$ spans $L_{4.12}$%
\hfil}%
\hbox to \myboxwidth{%
$22\infty222\infty22$%
\hfil}%
\box\matricesbox
\fi
}%
\hfill\discretionary{}{}{}%
\setbox\matricesbox=\hbox{%
{$\left[\!\llap{\phantom{%
\begingroup \smaller\smaller\smaller\begin{tabular}{@{}c@{}}%
\phantom{0}\\\phantom{0}\\\phantom{0}
\end{tabular}\endgroup%
}}\right.$}%
\begingroup \smaller\smaller\smaller\begin{tabular}{@{}c@{}}%
-1\\\phantom{0}\\\phantom{0}
\end{tabular}\endgroup%
\kern3pt%
\begingroup \smaller\smaller\smaller\begin{tabular}{@{}c@{}}%
\phantom{0}\\6\\-3
\end{tabular}\endgroup%
\kern3pt%
\begingroup \smaller\smaller\smaller\begin{tabular}{@{}c@{}}%
\phantom{0}\\-3\\6
\end{tabular}\endgroup%
{$\left.\llap{\phantom{%
\begingroup \smaller\smaller\smaller\begin{tabular}{@{}c@{}}%
\phantom{0}\\\phantom{0}\\\phantom{0}
\end{tabular}\endgroup%
}}\!\right]$}%
{$\left[\!\llap{\phantom{%
\begingroup \smaller\smaller\smaller\begin{tabular}{@{}c@{}}%
0\\0\\0
\end{tabular}\endgroup%
}}\right.$}%
\begingroup \smaller\smaller\smaller\begin{tabular}{@{}c@{}}%
2\\1\\1
\end{tabular}\endgroup%
\kern3pt%
\begingroup \smaller\smaller\smaller\begin{tabular}{@{}c@{}}%
2\\0\\1
\end{tabular}\endgroup%
\kern3pt%
\begingroup \smaller\smaller\smaller\begin{tabular}{@{}c@{}}%
2\\-1\\0
\end{tabular}\endgroup%
\kern3pt%
\begingroup \smaller\smaller\smaller\begin{tabular}{@{}c@{}}%
6\\-3\\-2
\end{tabular}\endgroup%
\kern3pt%
\begingroup \smaller\smaller\smaller\begin{tabular}{@{}c@{}}%
6\\-2\\-3
\end{tabular}\endgroup%
\kern3pt%
\begingroup \smaller\smaller\smaller\begin{tabular}{@{}c@{}}%
18\\-1\\-8
\end{tabular}\endgroup%
\kern3pt%
\begingroup \smaller\smaller\smaller\begin{tabular}{@{}c@{}}%
6\\1\\-2
\end{tabular}\endgroup%
\kern3pt%
\begingroup \smaller\smaller\smaller\begin{tabular}{@{}c@{}}%
18\\7\\-1
\end{tabular}\endgroup%
\kern3pt%
\begingroup \smaller\smaller\smaller\begin{tabular}{@{}c@{}}%
6\\3\\1
\end{tabular}\endgroup%
{$\left.\llap{\phantom{%
\begingroup \smaller\smaller\smaller\begin{tabular}{@{}c@{}}%
0\\0\\0
\end{tabular}\endgroup%
}}\!\right]$}%
}%
\ifdim\wd\matricesbox>\halfwidth\myboxwidth=\hsize\else\myboxwidth=\halfwidth\fi
\vbox{%
\ifdim\myboxwidth=\hsize
\setbox\onelinebox=\hbox{%
\vbox{\hbox{%
$\Pi_{9,15}$ spans $L_{155.1}$%
}\hbox{%
$332362622$%
}%
}%
\hfill\copy\matricesbox
}%
\ifdim\wd\onelinebox>\myboxwidth
\hbox to \myboxwidth{%
$\Pi_{9,15}$ spans $L_{155.1}$%
\hfil
$332362622$%
}%
\box\matricesbox
\else
\hbox to \myboxwidth{%
\unhbox\onelinebox
}%
\fi
\else
\hbox to \myboxwidth{%
$\Pi_{9,15}$ spans $L_{155.1}$%
\hfil}%
\hbox to \myboxwidth{%
$332362622$%
\hfil}%
\box\matricesbox
\fi
}%
\hfill\discretionary{}{}{}%
\setbox\matricesbox=\hbox{%
{$\left[\!\llap{\phantom{%
\begingroup \smaller\smaller\smaller\begin{tabular}{@{}c@{}}%
\phantom{0}\\\phantom{0}\\\phantom{0}
\end{tabular}\endgroup%
}}\right.$}%
\begingroup \smaller\smaller\smaller\begin{tabular}{@{}c@{}}%
-3\\\phantom{0}\\\phantom{0}
\end{tabular}\endgroup%
\kern3pt%
\begingroup \smaller\smaller\smaller\begin{tabular}{@{}c@{}}%
\phantom{0}\\4\\-2
\end{tabular}\endgroup%
\kern3pt%
\begingroup \smaller\smaller\smaller\begin{tabular}{@{}c@{}}%
\phantom{0}\\-2\\4
\end{tabular}\endgroup%
{$\left.\llap{\phantom{%
\begingroup \smaller\smaller\smaller\begin{tabular}{@{}c@{}}%
\phantom{0}\\\phantom{0}\\\phantom{0}
\end{tabular}\endgroup%
}}\!\right]$}%
{$\left[\!\llap{\phantom{%
\begingroup \smaller\smaller\smaller\begin{tabular}{@{}c@{}}%
0\\0\\0
\end{tabular}\endgroup%
}}\right.$}%
\begingroup \smaller\smaller\smaller\begin{tabular}{@{}c@{}}%
4\\1\\-3
\end{tabular}\endgroup%
\kern3pt%
\begingroup \smaller\smaller\smaller\begin{tabular}{@{}c@{}}%
4\\-1\\-4
\end{tabular}\endgroup%
\kern3pt%
\begingroup \smaller\smaller\smaller\begin{tabular}{@{}c@{}}%
1\\-1\\-1
\end{tabular}\endgroup%
\kern3pt%
\begingroup \smaller\smaller\smaller\begin{tabular}{@{}c@{}}%
1\\-1\\0
\end{tabular}\endgroup%
\kern3pt%
\begingroup \smaller\smaller\smaller\begin{tabular}{@{}c@{}}%
4\\-1\\3
\end{tabular}\endgroup%
\kern3pt%
\begingroup \smaller\smaller\smaller\begin{tabular}{@{}c@{}}%
4\\1\\4
\end{tabular}\endgroup%
\kern3pt%
\begingroup \smaller\smaller\smaller\begin{tabular}{@{}c@{}}%
1\\1\\1
\end{tabular}\endgroup%
\kern3pt%
\begingroup \smaller\smaller\smaller\begin{tabular}{@{}c@{}}%
4\\4\\1
\end{tabular}\endgroup%
\kern3pt%
\begingroup \smaller\smaller\smaller\begin{tabular}{@{}c@{}}%
4\\3\\-1
\end{tabular}\endgroup%
{$\left.\llap{\phantom{%
\begingroup \smaller\smaller\smaller\begin{tabular}{@{}c@{}}%
0\\0\\0
\end{tabular}\endgroup%
}}\!\right]$}%
}%
\ifdim\wd\matricesbox>\halfwidth\myboxwidth=\hsize\else\myboxwidth=\halfwidth\fi
\vbox{%
\ifdim\myboxwidth=\hsize
\setbox\onelinebox=\hbox{%
\vbox{\hbox{%
$\Pi_{9,16}=\hbox{GN}_{52}$ spans $L_{7.7}$%
}\hbox{%
$3\infty\infty\infty3\infty\infty3\infty$%
}%
}%
\hfill\copy\matricesbox
}%
\ifdim\wd\onelinebox>\myboxwidth
\hbox to \myboxwidth{%
$\Pi_{9,16}=\hbox{GN}_{52}$ spans $L_{7.7}$%
\hfil
$3\infty\infty\infty3\infty\infty3\infty$%
}%
\box\matricesbox
\else
\hbox to \myboxwidth{%
\unhbox\onelinebox
}%
\fi
\else
\hbox to \myboxwidth{%
$\Pi_{9,16}=\hbox{GN}_{52}$ spans $L_{7.7}$%
\hfil}%
\hbox to \myboxwidth{%
$3\infty\infty\infty3\infty\infty3\infty$%
\hfil}%
\box\matricesbox
\fi
}%
\hfill\discretionary{}{}{}%
\setbox\matricesbox=\hbox{%
{$\left[\!\llap{\phantom{%
\begingroup \smaller\smaller\smaller\begin{tabular}{@{}c@{}}%
\phantom{0}\\\phantom{0}\\\phantom{0}
\end{tabular}\endgroup%
}}\right.$}%
\begingroup \smaller\smaller\smaller\begin{tabular}{@{}c@{}}%
-1\\\phantom{0}\\\phantom{0}
\end{tabular}\endgroup%
\kern3pt%
\begingroup \smaller\smaller\smaller\begin{tabular}{@{}c@{}}%
\phantom{0}\\6\\-3
\end{tabular}\endgroup%
\kern3pt%
\begingroup \smaller\smaller\smaller\begin{tabular}{@{}c@{}}%
\phantom{0}\\-3\\6
\end{tabular}\endgroup%
{$\left.\llap{\phantom{%
\begingroup \smaller\smaller\smaller\begin{tabular}{@{}c@{}}%
\phantom{0}\\\phantom{0}\\\phantom{0}
\end{tabular}\endgroup%
}}\!\right]$}%
{$\left[\!\llap{\phantom{%
\begingroup \smaller\smaller\smaller\begin{tabular}{@{}c@{}}%
0\\0\\0
\end{tabular}\endgroup%
}}\right.$}%
\begingroup \smaller\smaller\smaller\begin{tabular}{@{}c@{}}%
6\\-3\\-2
\end{tabular}\endgroup%
\kern3pt%
\begingroup \smaller\smaller\smaller\begin{tabular}{@{}c@{}}%
6\\-2\\-3
\end{tabular}\endgroup%
\kern3pt%
\begingroup \smaller\smaller\smaller\begin{tabular}{@{}c@{}}%
2\\0\\-1
\end{tabular}\endgroup%
\kern3pt%
\begingroup \smaller\smaller\smaller\begin{tabular}{@{}c@{}}%
2\\1\\0
\end{tabular}\endgroup%
\kern3pt%
\begingroup \smaller\smaller\smaller\begin{tabular}{@{}c@{}}%
6\\3\\2
\end{tabular}\endgroup%
\kern3pt%
\begingroup \smaller\smaller\smaller\begin{tabular}{@{}c@{}}%
6\\2\\3
\end{tabular}\endgroup%
\kern3pt%
\begingroup \smaller\smaller\smaller\begin{tabular}{@{}c@{}}%
2\\0\\1
\end{tabular}\endgroup%
\kern3pt%
\begingroup \smaller\smaller\smaller\begin{tabular}{@{}c@{}}%
6\\-2\\1
\end{tabular}\endgroup%
\kern3pt%
\begingroup \smaller\smaller\smaller\begin{tabular}{@{}c@{}}%
18\\-8\\-1
\end{tabular}\endgroup%
{$\left.\llap{\phantom{%
\begingroup \smaller\smaller\smaller\begin{tabular}{@{}c@{}}%
0\\0\\0
\end{tabular}\endgroup%
}}\!\right]$}%
}%
\ifdim\wd\matricesbox>\halfwidth\myboxwidth=\hsize\else\myboxwidth=\halfwidth\fi
\vbox{%
\ifdim\myboxwidth=\hsize
\setbox\onelinebox=\hbox{%
\vbox{\hbox{%
$\Pi_{9,17}$ spans $L_{168.1}$%
}\hbox{%
$323232226$%
}%
}%
\hfill\copy\matricesbox
}%
\ifdim\wd\onelinebox>\myboxwidth
\hbox to \myboxwidth{%
$\Pi_{9,17}$ spans $L_{168.1}$%
\hfil
$323232226$%
}%
\box\matricesbox
\else
\hbox to \myboxwidth{%
\unhbox\onelinebox
}%
\fi
\else
\hbox to \myboxwidth{%
$\Pi_{9,17}$ spans $L_{168.1}$%
\hfil}%
\hbox to \myboxwidth{%
$323232226$%
\hfil}%
\box\matricesbox
\fi
}%
\hfill\discretionary{}{}{}%
\setbox\matricesbox=\hbox{%
{$\left[\!\llap{\phantom{%
\begingroup \smaller\smaller\smaller\begin{tabular}{@{}c@{}}%
\phantom{0}\\\phantom{0}\\\phantom{0}
\end{tabular}\endgroup%
}}\right.$}%
\begingroup \smaller\smaller\smaller\begin{tabular}{@{}c@{}}%
-1\\\phantom{0}\\\phantom{0}
\end{tabular}\endgroup%
\kern3pt%
\begingroup \smaller\smaller\smaller\begin{tabular}{@{}c@{}}%
\phantom{0}\\6\\-3
\end{tabular}\endgroup%
\kern3pt%
\begingroup \smaller\smaller\smaller\begin{tabular}{@{}c@{}}%
\phantom{0}\\-3\\6
\end{tabular}\endgroup%
{$\left.\llap{\phantom{%
\begingroup \smaller\smaller\smaller\begin{tabular}{@{}c@{}}%
\phantom{0}\\\phantom{0}\\\phantom{0}
\end{tabular}\endgroup%
}}\!\right]$}%
{$\left[\!\llap{\phantom{%
\begingroup \smaller\smaller\smaller\begin{tabular}{@{}c@{}}%
0\\0\\0
\end{tabular}\endgroup%
}}\right.$}%
\begingroup \smaller\smaller\smaller\begin{tabular}{@{}c@{}}%
6\\-1\\-3
\end{tabular}\endgroup%
\kern3pt%
\begingroup \smaller\smaller\smaller\begin{tabular}{@{}c@{}}%
6\\1\\-2
\end{tabular}\endgroup%
\kern3pt%
\begingroup \smaller\smaller\smaller\begin{tabular}{@{}c@{}}%
2\\1\\0
\end{tabular}\endgroup%
\kern3pt%
\begingroup \smaller\smaller\smaller\begin{tabular}{@{}c@{}}%
2\\1\\1
\end{tabular}\endgroup%
\kern3pt%
\begingroup \smaller\smaller\smaller\begin{tabular}{@{}c@{}}%
6\\1\\3
\end{tabular}\endgroup%
\kern3pt%
\begingroup \smaller\smaller\smaller\begin{tabular}{@{}c@{}}%
18\\-1\\7
\end{tabular}\endgroup%
\kern3pt%
\begingroup \smaller\smaller\smaller\begin{tabular}{@{}c@{}}%
6\\-2\\1
\end{tabular}\endgroup%
\kern3pt%
\begingroup \smaller\smaller\smaller\begin{tabular}{@{}c@{}}%
6\\-3\\-1
\end{tabular}\endgroup%
\kern3pt%
\begingroup \smaller\smaller\smaller\begin{tabular}{@{}c@{}}%
2\\-1\\-1
\end{tabular}\endgroup%
{$\left.\llap{\phantom{%
\begingroup \smaller\smaller\smaller\begin{tabular}{@{}c@{}}%
0\\0\\0
\end{tabular}\endgroup%
}}\!\right]$}%
}%
\ifdim\wd\matricesbox>\halfwidth\myboxwidth=\hsize\else\myboxwidth=\halfwidth\fi
\vbox{%
\ifdim\myboxwidth=\hsize
\setbox\onelinebox=\hbox{%
\vbox{\hbox{%
$\Pi_{9,18}$ spans $L_{155.1}$%
}\hbox{%
$323226322$%
}%
}%
\hfill\copy\matricesbox
}%
\ifdim\wd\onelinebox>\myboxwidth
\hbox to \myboxwidth{%
$\Pi_{9,18}$ spans $L_{155.1}$%
\hfil
$323226322$%
}%
\box\matricesbox
\else
\hbox to \myboxwidth{%
\unhbox\onelinebox
}%
\fi
\else
\hbox to \myboxwidth{%
$\Pi_{9,18}$ spans $L_{155.1}$%
\hfil}%
\hbox to \myboxwidth{%
$323226322$%
\hfil}%
\box\matricesbox
\fi
}%
\hfill\discretionary{}{}{}%

\vskip2pt\hrule\vskip2pt

\leavevmode\setbox\matricesbox=\hbox{%
{$\left[\!\llap{\phantom{%
\begingroup \smaller\smaller\smaller\begin{tabular}{@{}c@{}}%
\phantom{0}\\\phantom{0}\\\phantom{0}
\end{tabular}\endgroup%
}}\right.$}%
\begingroup \smaller\smaller\smaller\begin{tabular}{@{}c@{}}%
-1/4\\\phantom{0}\\\phantom{0}
\end{tabular}\endgroup%
\kern3pt%
\begingroup \smaller\smaller\smaller\begin{tabular}{@{}c@{}}%
\phantom{0}\\15/2\\\phantom{0}
\end{tabular}\endgroup%
\kern3pt%
\begingroup \smaller\smaller\smaller\begin{tabular}{@{}c@{}}%
\phantom{0}\\\phantom{0}\\55/2
\end{tabular}\endgroup%
{$\left.\llap{\phantom{%
\begingroup \smaller\smaller\smaller\begin{tabular}{@{}c@{}}%
\phantom{0}\\\phantom{0}\\\phantom{0}
\end{tabular}\endgroup%
}}\!\right]$}%
{$\left[\!\llap{\phantom{%
\begingroup \smaller\smaller\smaller\begin{tabular}{@{}c@{}}%
0\\0\\0
\end{tabular}\endgroup%
}}\right.$}%
\begingroup \smaller\smaller\smaller\begin{tabular}{@{}c@{}}%
20\\-4\\0
\end{tabular}\endgroup%
\kern3pt%
\begingroup \smaller\smaller\smaller\begin{tabular}{@{}c@{}}%
66\\-11\\3
\end{tabular}\endgroup%
\kern3pt%
\begingroup \smaller\smaller\smaller\begin{tabular}{@{}c@{}}%
10\\-1\\1
\end{tabular}\endgroup%
{$\left.\llap{\phantom{%
\begingroup \smaller\smaller\smaller\begin{tabular}{@{}c@{}}%
0\\0\\0
\end{tabular}\endgroup%
}}\!\right]$}%
}%
\ifdim\wd\matricesbox>\halfwidth\myboxwidth=\hsize\else\myboxwidth=\halfwidth\fi
\vbox{%
\ifdim\myboxwidth=\hsize
\setbox\onelinebox=\hbox{%
\vbox{\hbox{%
$\Pi_{10,1}$ spans $L_{78.9}$%
}\hbox{%
$22|22\slashthree22|22\slashthree\rtimes D_{4}$%
}%
}%
\hfill\copy\matricesbox
}%
\ifdim\wd\onelinebox>\myboxwidth
\hbox to \myboxwidth{%
$\Pi_{10,1}$ spans $L_{78.9}$%
\hfil
$22|22\slashthree22|22\slashthree\rtimes D_{4}$%
}%
\box\matricesbox
\else
\hbox to \myboxwidth{%
\unhbox\onelinebox
}%
\fi
\else
\hbox to \myboxwidth{%
$\Pi_{10,1}$ spans $L_{78.9}$%
\hfil}%
\hbox to \myboxwidth{%
$22|22\slashthree22|22\slashthree\rtimes D_{4}$%
\hfil}%
\box\matricesbox
\fi
}%
\hfill\discretionary{}{}{}%
\setbox\matricesbox=\hbox{%
{$\left[\!\llap{\phantom{%
\begingroup \smaller\smaller\smaller\begin{tabular}{@{}c@{}}%
\phantom{0}\\\phantom{0}\\\phantom{0}
\end{tabular}\endgroup%
}}\right.$}%
\begingroup \smaller\smaller\smaller\begin{tabular}{@{}c@{}}%
-1/2\\\phantom{0}\\\phantom{0}
\end{tabular}\endgroup%
\kern3pt%
\begingroup \smaller\smaller\smaller\begin{tabular}{@{}c@{}}%
\phantom{0}\\15/2\\\phantom{0}
\end{tabular}\endgroup%
\kern3pt%
\begingroup \smaller\smaller\smaller\begin{tabular}{@{}c@{}}%
\phantom{0}\\\phantom{0}\\30
\end{tabular}\endgroup%
{$\left.\llap{\phantom{%
\begingroup \smaller\smaller\smaller\begin{tabular}{@{}c@{}}%
\phantom{0}\\\phantom{0}\\\phantom{0}
\end{tabular}\endgroup%
}}\!\right]$}%
{$\left[\!\llap{\phantom{%
\begingroup \smaller\smaller\smaller\begin{tabular}{@{}c@{}}%
0\\0\\0
\end{tabular}\endgroup%
}}\right.$}%
\begingroup \smaller\smaller\smaller\begin{tabular}{@{}c@{}}%
3\\-1\\0
\end{tabular}\endgroup%
\kern3pt%
\begingroup \smaller\smaller\smaller\begin{tabular}{@{}c@{}}%
10\\-2\\-1
\end{tabular}\endgroup%
\kern3pt%
\begingroup \smaller\smaller\smaller\begin{tabular}{@{}c@{}}%
15\\-1\\-2
\end{tabular}\endgroup%
{$\left.\llap{\phantom{%
\begingroup \smaller\smaller\smaller\begin{tabular}{@{}c@{}}%
0\\0\\0
\end{tabular}\endgroup%
}}\!\right]$}%
}%
\ifdim\wd\matricesbox>\halfwidth\myboxwidth=\hsize\else\myboxwidth=\halfwidth\fi
\vbox{%
\ifdim\myboxwidth=\hsize
\setbox\onelinebox=\hbox{%
\vbox{\hbox{%
$\Pi_{10,2}$ spans $L_{127.15}$%
}\hbox{%
$|22\slashtwo22|22\slashtwo22\rtimes D_{4}$%
}%
}%
\hfill\copy\matricesbox
}%
\ifdim\wd\onelinebox>\myboxwidth
\hbox to \myboxwidth{%
$\Pi_{10,2}$ spans $L_{127.15}$%
\hfil
$|22\slashtwo22|22\slashtwo22\rtimes D_{4}$%
}%
\box\matricesbox
\else
\hbox to \myboxwidth{%
\unhbox\onelinebox
}%
\fi
\else
\hbox to \myboxwidth{%
$\Pi_{10,2}$ spans $L_{127.15}$%
\hfil}%
\hbox to \myboxwidth{%
$|22\slashtwo22|22\slashtwo22\rtimes D_{4}$%
\hfil}%
\box\matricesbox
\fi
}%
\hfill\discretionary{}{}{}%
\setbox\matricesbox=\hbox{%
{$\left[\!\llap{\phantom{%
\begingroup \smaller\smaller\smaller\begin{tabular}{@{}c@{}}%
\phantom{0}\\\phantom{0}\\\phantom{0}
\end{tabular}\endgroup%
}}\right.$}%
\begingroup \smaller\smaller\smaller\begin{tabular}{@{}c@{}}%
-1\\\phantom{0}\\\phantom{0}
\end{tabular}\endgroup%
\kern3pt%
\begingroup \smaller\smaller\smaller\begin{tabular}{@{}c@{}}%
\phantom{0}\\3/2\\\phantom{0}
\end{tabular}\endgroup%
\kern3pt%
\begingroup \smaller\smaller\smaller\begin{tabular}{@{}c@{}}%
\phantom{0}\\\phantom{0}\\21/2
\end{tabular}\endgroup%
{$\left.\llap{\phantom{%
\begingroup \smaller\smaller\smaller\begin{tabular}{@{}c@{}}%
\phantom{0}\\\phantom{0}\\\phantom{0}
\end{tabular}\endgroup%
}}\!\right]$}%
{$\left[\!\llap{\phantom{%
\begingroup \smaller\smaller\smaller\begin{tabular}{@{}c@{}}%
0\\0\\0
\end{tabular}\endgroup%
}}\right.$}%
\begingroup \smaller\smaller\smaller\begin{tabular}{@{}c@{}}%
2\\-2\\0
\end{tabular}\endgroup%
\kern3pt%
\begingroup \smaller\smaller\smaller\begin{tabular}{@{}c@{}}%
21\\-14\\-4
\end{tabular}\endgroup%
\kern3pt%
\begingroup \smaller\smaller\smaller\begin{tabular}{@{}c@{}}%
3\\-1\\-1
\end{tabular}\endgroup%
{$\left.\llap{\phantom{%
\begingroup \smaller\smaller\smaller\begin{tabular}{@{}c@{}}%
0\\0\\0
\end{tabular}\endgroup%
}}\!\right]$}%
}%
\ifdim\wd\matricesbox>\halfwidth\myboxwidth=\hsize\else\myboxwidth=\halfwidth\fi
\vbox{%
\ifdim\myboxwidth=\hsize
\setbox\onelinebox=\hbox{%
\vbox{\hbox{%
$\Pi_{10,3}$ spans $L_{24.7}$%
}\hbox{%
$22|22\slashtwo22|22\slashtwo\rtimes D_{4}$%
}%
}%
\hfill\copy\matricesbox
}%
\ifdim\wd\onelinebox>\myboxwidth
\hbox to \myboxwidth{%
$\Pi_{10,3}$ spans $L_{24.7}$%
\hfil
$22|22\slashtwo22|22\slashtwo\rtimes D_{4}$%
}%
\box\matricesbox
\else
\hbox to \myboxwidth{%
\unhbox\onelinebox
}%
\fi
\else
\hbox to \myboxwidth{%
$\Pi_{10,3}$ spans $L_{24.7}$%
\hfil}%
\hbox to \myboxwidth{%
$22|22\slashtwo22|22\slashtwo\rtimes D_{4}$%
\hfil}%
\box\matricesbox
\fi
}%
\hfill\discretionary{}{}{}%
\setbox\matricesbox=\hbox{%
{$\left[\!\llap{\phantom{%
\begingroup \smaller\smaller\smaller\begin{tabular}{@{}c@{}}%
\phantom{0}\\\phantom{0}\\\phantom{0}
\end{tabular}\endgroup%
}}\right.$}%
\begingroup \smaller\smaller\smaller\begin{tabular}{@{}c@{}}%
-1/2\\\phantom{0}\\\phantom{0}
\end{tabular}\endgroup%
\kern3pt%
\begingroup \smaller\smaller\smaller\begin{tabular}{@{}c@{}}%
\phantom{0}\\7\\\phantom{0}
\end{tabular}\endgroup%
\kern3pt%
\begingroup \smaller\smaller\smaller\begin{tabular}{@{}c@{}}%
\phantom{0}\\\phantom{0}\\1/2
\end{tabular}\endgroup%
{$\left.\llap{\phantom{%
\begingroup \smaller\smaller\smaller\begin{tabular}{@{}c@{}}%
\phantom{0}\\\phantom{0}\\\phantom{0}
\end{tabular}\endgroup%
}}\!\right]$}%
{$\left[\!\llap{\phantom{%
\begingroup \smaller\smaller\smaller\begin{tabular}{@{}c@{}}%
0\\0\\0
\end{tabular}\endgroup%
}}\right.$}%
\begingroup \smaller\smaller\smaller\begin{tabular}{@{}c@{}}%
14\\4\\0
\end{tabular}\endgroup%
\kern3pt%
\begingroup \smaller\smaller\smaller\begin{tabular}{@{}c@{}}%
16\\4\\-8
\end{tabular}\endgroup%
\kern3pt%
\begingroup \smaller\smaller\smaller\begin{tabular}{@{}c@{}}%
7\\1\\-7
\end{tabular}\endgroup%
{$\left.\llap{\phantom{%
\begingroup \smaller\smaller\smaller\begin{tabular}{@{}c@{}}%
0\\0\\0
\end{tabular}\endgroup%
}}\!\right]$}%
}%
\ifdim\wd\matricesbox>\halfwidth\myboxwidth=\hsize\else\myboxwidth=\halfwidth\fi
\vbox{%
\ifdim\myboxwidth=\hsize
\setbox\onelinebox=\hbox{%
\vbox{\hbox{%
$\Pi_{10,4}$ spans $L_{9.6}$%
}\hbox{%
$22|22\slashinfty22|22\slashinfty\rtimes D_{4}$%
}%
}%
\hfill\copy\matricesbox
}%
\ifdim\wd\onelinebox>\myboxwidth
\hbox to \myboxwidth{%
$\Pi_{10,4}$ spans $L_{9.6}$%
\hfil
$22|22\slashinfty22|22\slashinfty\rtimes D_{4}$%
}%
\box\matricesbox
\else
\hbox to \myboxwidth{%
\unhbox\onelinebox
}%
\fi
\else
\hbox to \myboxwidth{%
$\Pi_{10,4}$ spans $L_{9.6}$%
\hfil}%
\hbox to \myboxwidth{%
$22|22\slashinfty22|22\slashinfty\rtimes D_{4}$%
\hfil}%
\box\matricesbox
\fi
}%
\hfill\discretionary{}{}{}%
\setbox\matricesbox=\hbox{%
{$\left[\!\llap{\phantom{%
\begingroup \smaller\smaller\smaller\begin{tabular}{@{}c@{}}%
\phantom{0}\\\phantom{0}\\\phantom{0}
\end{tabular}\endgroup%
}}\right.$}%
\begingroup \smaller\smaller\smaller\begin{tabular}{@{}c@{}}%
-1/8\\\phantom{0}\\\phantom{0}
\end{tabular}\endgroup%
\kern3pt%
\begingroup \smaller\smaller\smaller\begin{tabular}{@{}c@{}}%
\phantom{0}\\10\\\phantom{0}
\end{tabular}\endgroup%
\kern3pt%
\begingroup \smaller\smaller\smaller\begin{tabular}{@{}c@{}}%
\phantom{0}\\\phantom{0}\\25/2
\end{tabular}\endgroup%
{$\left.\llap{\phantom{%
\begingroup \smaller\smaller\smaller\begin{tabular}{@{}c@{}}%
\phantom{0}\\\phantom{0}\\\phantom{0}
\end{tabular}\endgroup%
}}\!\right]$}%
{$\left[\!\llap{\phantom{%
\begingroup \smaller\smaller\smaller\begin{tabular}{@{}c@{}}%
0\\0\\0
\end{tabular}\endgroup%
}}\right.$}%
\begingroup \smaller\smaller\smaller\begin{tabular}{@{}c@{}}%
32\\-4\\0
\end{tabular}\endgroup%
\kern3pt%
\begingroup \smaller\smaller\smaller\begin{tabular}{@{}c@{}}%
50\\-5\\-3
\end{tabular}\endgroup%
\kern3pt%
\begingroup \smaller\smaller\smaller\begin{tabular}{@{}c@{}}%
40\\-2\\-4
\end{tabular}\endgroup%
{$\left.\llap{\phantom{%
\begingroup \smaller\smaller\smaller\begin{tabular}{@{}c@{}}%
0\\0\\0
\end{tabular}\endgroup%
}}\!\right]$}%
}%
\ifdim\wd\matricesbox>\halfwidth\myboxwidth=\hsize\else\myboxwidth=\halfwidth\fi
\vbox{%
\ifdim\myboxwidth=\hsize
\setbox\onelinebox=\hbox{%
\vbox{\hbox{%
$\Pi_{10,5}$ spans $L_{10.1}$%
}\hbox{%
$22|22\slashinfty22|22\slashinfty\rtimes D_{4}$%
}%
}%
\hfill\copy\matricesbox
}%
\ifdim\wd\onelinebox>\myboxwidth
\hbox to \myboxwidth{%
$\Pi_{10,5}$ spans $L_{10.1}$%
\hfil
$22|22\slashinfty22|22\slashinfty\rtimes D_{4}$%
}%
\box\matricesbox
\else
\hbox to \myboxwidth{%
\unhbox\onelinebox
}%
\fi
\else
\hbox to \myboxwidth{%
$\Pi_{10,5}$ spans $L_{10.1}$%
\hfil}%
\hbox to \myboxwidth{%
$22|22\slashinfty22|22\slashinfty\rtimes D_{4}$%
\hfil}%
\box\matricesbox
\fi
}%
\hfill\discretionary{}{}{}%
\setbox\matricesbox=\hbox{%
{$\left[\!\llap{\phantom{%
\begingroup \smaller\smaller\smaller\begin{tabular}{@{}c@{}}%
\phantom{0}\\\phantom{0}\\\phantom{0}
\end{tabular}\endgroup%
}}\right.$}%
\begingroup \smaller\smaller\smaller\begin{tabular}{@{}c@{}}%
-1\\\phantom{0}\\\phantom{0}
\end{tabular}\endgroup%
\kern3pt%
\begingroup \smaller\smaller\smaller\begin{tabular}{@{}c@{}}%
\phantom{0}\\1/2\\\phantom{0}
\end{tabular}\endgroup%
\kern3pt%
\begingroup \smaller\smaller\smaller\begin{tabular}{@{}c@{}}%
\phantom{0}\\\phantom{0}\\75/2
\end{tabular}\endgroup%
{$\left.\llap{\phantom{%
\begingroup \smaller\smaller\smaller\begin{tabular}{@{}c@{}}%
\phantom{0}\\\phantom{0}\\\phantom{0}
\end{tabular}\endgroup%
}}\!\right]$}%
{$\left[\!\llap{\phantom{%
\begingroup \smaller\smaller\smaller\begin{tabular}{@{}c@{}}%
0\\0\\0
\end{tabular}\endgroup%
}}\right.$}%
\begingroup \smaller\smaller\smaller\begin{tabular}{@{}c@{}}%
1\\-2\\0
\end{tabular}\endgroup%
\kern3pt%
\begingroup \smaller\smaller\smaller\begin{tabular}{@{}c@{}}%
25\\-25\\3
\end{tabular}\endgroup%
\kern3pt%
\begingroup \smaller\smaller\smaller\begin{tabular}{@{}c@{}}%
6\\-3\\1
\end{tabular}\endgroup%
{$\left.\llap{\phantom{%
\begingroup \smaller\smaller\smaller\begin{tabular}{@{}c@{}}%
0\\0\\0
\end{tabular}\endgroup%
}}\!\right]$}%
}%
\ifdim\wd\matricesbox>\halfwidth\myboxwidth=\hsize\else\myboxwidth=\halfwidth\fi
\vbox{%
\ifdim\myboxwidth=\hsize
\setbox\onelinebox=\hbox{%
\vbox{\hbox{%
$\Pi_{10,6}$ spans $L_{186.2}$%
}\hbox{%
$|22\slashthree22|22\slashthree22\rtimes D_{4}$%
}%
}%
\hfill\copy\matricesbox
}%
\ifdim\wd\onelinebox>\myboxwidth
\hbox to \myboxwidth{%
$\Pi_{10,6}$ spans $L_{186.2}$%
\hfil
$|22\slashthree22|22\slashthree22\rtimes D_{4}$%
}%
\box\matricesbox
\else
\hbox to \myboxwidth{%
\unhbox\onelinebox
}%
\fi
\else
\hbox to \myboxwidth{%
$\Pi_{10,6}$ spans $L_{186.2}$%
\hfil}%
\hbox to \myboxwidth{%
$|22\slashthree22|22\slashthree22\rtimes D_{4}$%
\hfil}%
\box\matricesbox
\fi
}%
\hfill\discretionary{}{}{}%
\setbox\matricesbox=\hbox{%
{$\left[\!\llap{\phantom{%
\begingroup \smaller\smaller\smaller\begin{tabular}{@{}c@{}}%
\phantom{0}\\\phantom{0}\\\phantom{0}
\end{tabular}\endgroup%
}}\right.$}%
\begingroup \smaller\smaller\smaller\begin{tabular}{@{}c@{}}%
-1\\\phantom{0}\\\phantom{0}
\end{tabular}\endgroup%
\kern3pt%
\begingroup \smaller\smaller\smaller\begin{tabular}{@{}c@{}}%
\phantom{0}\\2\\\phantom{0}
\end{tabular}\endgroup%
\kern3pt%
\begingroup \smaller\smaller\smaller\begin{tabular}{@{}c@{}}%
\phantom{0}\\\phantom{0}\\16
\end{tabular}\endgroup%
{$\left.\llap{\phantom{%
\begingroup \smaller\smaller\smaller\begin{tabular}{@{}c@{}}%
\phantom{0}\\\phantom{0}\\\phantom{0}
\end{tabular}\endgroup%
}}\!\right]$}%
{$\left[\!\llap{\phantom{%
\begingroup \smaller\smaller\smaller\begin{tabular}{@{}c@{}}%
0\\0\\0
\end{tabular}\endgroup%
}}\right.$}%
\begingroup \smaller\smaller\smaller\begin{tabular}{@{}c@{}}%
1\\-1\\0
\end{tabular}\endgroup%
\kern3pt%
\begingroup \smaller\smaller\smaller\begin{tabular}{@{}c@{}}%
16\\-8\\-3
\end{tabular}\endgroup%
\kern3pt%
\begingroup \smaller\smaller\smaller\begin{tabular}{@{}c@{}}%
8\\-2\\-2
\end{tabular}\endgroup%
{$\left.\llap{\phantom{%
\begingroup \smaller\smaller\smaller\begin{tabular}{@{}c@{}}%
0\\0\\0
\end{tabular}\endgroup%
}}\!\right]$}%
}%
\ifdim\wd\matricesbox>\halfwidth\myboxwidth=\hsize\else\myboxwidth=\halfwidth\fi
\vbox{%
\ifdim\myboxwidth=\hsize
\setbox\onelinebox=\hbox{%
\vbox{\hbox{%
$\Pi_{10,7}$ spans $L_{141.10}$%
}\hbox{%
$|22\slashinfty22|22\slashinfty22\rtimes D_{4}$%
}%
}%
\hfill\copy\matricesbox
}%
\ifdim\wd\onelinebox>\myboxwidth
\hbox to \myboxwidth{%
$\Pi_{10,7}$ spans $L_{141.10}$%
\hfil
$|22\slashinfty22|22\slashinfty22\rtimes D_{4}$%
}%
\box\matricesbox
\else
\hbox to \myboxwidth{%
\unhbox\onelinebox
}%
\fi
\else
\hbox to \myboxwidth{%
$\Pi_{10,7}$ spans $L_{141.10}$%
\hfil}%
\hbox to \myboxwidth{%
$|22\slashinfty22|22\slashinfty22\rtimes D_{4}$%
\hfil}%
\box\matricesbox
\fi
}%
\hfill\discretionary{}{}{}%
\setbox\matricesbox=\hbox{%
{$\left[\!\llap{\phantom{%
\begingroup \smaller\smaller\smaller\begin{tabular}{@{}c@{}}%
\phantom{0}\\\phantom{0}\\\phantom{0}
\end{tabular}\endgroup%
}}\right.$}%
\begingroup \smaller\smaller\smaller\begin{tabular}{@{}c@{}}%
-1\\\phantom{0}\\\phantom{0}
\end{tabular}\endgroup%
\kern3pt%
\begingroup \smaller\smaller\smaller\begin{tabular}{@{}c@{}}%
\phantom{0}\\6\\\phantom{0}
\end{tabular}\endgroup%
\kern3pt%
\begingroup \smaller\smaller\smaller\begin{tabular}{@{}c@{}}%
\phantom{0}\\\phantom{0}\\36
\end{tabular}\endgroup%
{$\left.\llap{\phantom{%
\begingroup \smaller\smaller\smaller\begin{tabular}{@{}c@{}}%
\phantom{0}\\\phantom{0}\\\phantom{0}
\end{tabular}\endgroup%
}}\!\right]$}%
{$\left[\!\llap{\phantom{%
\begingroup \smaller\smaller\smaller\begin{tabular}{@{}c@{}}%
0\\0\\0
\end{tabular}\endgroup%
}}\right.$}%
\begingroup \smaller\smaller\smaller\begin{tabular}{@{}c@{}}%
2\\-1\\0
\end{tabular}\endgroup%
\kern3pt%
\begingroup \smaller\smaller\smaller\begin{tabular}{@{}c@{}}%
9\\-3\\-1
\end{tabular}\endgroup%
\kern3pt%
\begingroup \smaller\smaller\smaller\begin{tabular}{@{}c@{}}%
6\\-1\\-1
\end{tabular}\endgroup%
{$\left.\llap{\phantom{%
\begingroup \smaller\smaller\smaller\begin{tabular}{@{}c@{}}%
0\\0\\0
\end{tabular}\endgroup%
}}\!\right]$}%
}%
\ifdim\wd\matricesbox>\halfwidth\myboxwidth=\hsize\else\myboxwidth=\halfwidth\fi
\vbox{%
\ifdim\myboxwidth=\hsize
\setbox\onelinebox=\hbox{%
\vbox{\hbox{%
$\Pi_{10,8}$ spans $L_{150.14}$%
}\hbox{%
$|22\slashinfty22|22\slashinfty22\rtimes D_{4}$%
}%
}%
\hfill\copy\matricesbox
}%
\ifdim\wd\onelinebox>\myboxwidth
\hbox to \myboxwidth{%
$\Pi_{10,8}$ spans $L_{150.14}$%
\hfil
$|22\slashinfty22|22\slashinfty22\rtimes D_{4}$%
}%
\box\matricesbox
\else
\hbox to \myboxwidth{%
\unhbox\onelinebox
}%
\fi
\else
\hbox to \myboxwidth{%
$\Pi_{10,8}$ spans $L_{150.14}$%
\hfil}%
\hbox to \myboxwidth{%
$|22\slashinfty22|22\slashinfty22\rtimes D_{4}$%
\hfil}%
\box\matricesbox
\fi
}%
\hfill\discretionary{}{}{}%
\setbox\matricesbox=\hbox{%
{$\left[\!\llap{\phantom{%
\begingroup \smaller\smaller\smaller\begin{tabular}{@{}c@{}}%
\phantom{0}\\\phantom{0}\\\phantom{0}
\end{tabular}\endgroup%
}}\right.$}%
\begingroup \smaller\smaller\smaller\begin{tabular}{@{}c@{}}%
-1\\\phantom{0}\\\phantom{0}
\end{tabular}\endgroup%
\kern3pt%
\begingroup \smaller\smaller\smaller\begin{tabular}{@{}c@{}}%
\phantom{0}\\3\\\phantom{0}
\end{tabular}\endgroup%
\kern3pt%
\begingroup \smaller\smaller\smaller\begin{tabular}{@{}c@{}}%
\phantom{0}\\\phantom{0}\\3
\end{tabular}\endgroup%
{$\left.\llap{\phantom{%
\begingroup \smaller\smaller\smaller\begin{tabular}{@{}c@{}}%
\phantom{0}\\\phantom{0}\\\phantom{0}
\end{tabular}\endgroup%
}}\!\right]$}%
{$\left[\!\llap{\phantom{%
\begingroup \smaller\smaller\smaller\begin{tabular}{@{}c@{}}%
0\\0\\0
\end{tabular}\endgroup%
}}\right.$}%
\begingroup \smaller\smaller\smaller\begin{tabular}{@{}c@{}}%
24\\-14\\-2
\end{tabular}\endgroup%
\kern3pt%
\begingroup \smaller\smaller\smaller\begin{tabular}{@{}c@{}}%
2\\-1\\-1
\end{tabular}\endgroup%
\kern3pt%
\begingroup \smaller\smaller\smaller\begin{tabular}{@{}c@{}}%
3\\0\\-2
\end{tabular}\endgroup%
{$\left.\llap{\phantom{%
\begingroup \smaller\smaller\smaller\begin{tabular}{@{}c@{}}%
0\\0\\0
\end{tabular}\endgroup%
}}\!\right]$}%
}%
\ifdim\wd\matricesbox>\halfwidth\myboxwidth=\hsize\else\myboxwidth=\halfwidth\fi
\vbox{%
\ifdim\myboxwidth=\hsize
\setbox\onelinebox=\hbox{%
\vbox{\hbox{%
$\Pi_{10,9}$ spans $L_{123.8}$%
}\hbox{%
$2\slashtwo22|22\slashtwo22|2\rtimes D_{4}$%
}%
}%
\hfill\copy\matricesbox
}%
\ifdim\wd\onelinebox>\myboxwidth
\hbox to \myboxwidth{%
$\Pi_{10,9}$ spans $L_{123.8}$%
\hfil
$2\slashtwo22|22\slashtwo22|2\rtimes D_{4}$%
}%
\box\matricesbox
\else
\hbox to \myboxwidth{%
\unhbox\onelinebox
}%
\fi
\else
\hbox to \myboxwidth{%
$\Pi_{10,9}$ spans $L_{123.8}$%
\hfil}%
\hbox to \myboxwidth{%
$2\slashtwo22|22\slashtwo22|2\rtimes D_{4}$%
\hfil}%
\box\matricesbox
\fi
}%
\hfill\discretionary{}{}{}%
\setbox\matricesbox=\hbox{%
{$\left[\!\llap{\phantom{%
\begingroup \smaller\smaller\smaller\begin{tabular}{@{}c@{}}%
\phantom{0}\\\phantom{0}\\\phantom{0}
\end{tabular}\endgroup%
}}\right.$}%
\begingroup \smaller\smaller\smaller\begin{tabular}{@{}c@{}}%
-1\\\phantom{0}\\\phantom{0}
\end{tabular}\endgroup%
\kern3pt%
\begingroup \smaller\smaller\smaller\begin{tabular}{@{}c@{}}%
\phantom{0}\\6\\\phantom{0}
\end{tabular}\endgroup%
\kern3pt%
\begingroup \smaller\smaller\smaller\begin{tabular}{@{}c@{}}%
\phantom{0}\\\phantom{0}\\6
\end{tabular}\endgroup%
{$\left.\llap{\phantom{%
\begingroup \smaller\smaller\smaller\begin{tabular}{@{}c@{}}%
\phantom{0}\\\phantom{0}\\\phantom{0}
\end{tabular}\endgroup%
}}\!\right]$}%
{$\left[\!\llap{\phantom{%
\begingroup \smaller\smaller\smaller\begin{tabular}{@{}c@{}}%
0\\0\\0
\end{tabular}\endgroup%
}}\right.$}%
\begingroup \smaller\smaller\smaller\begin{tabular}{@{}c@{}}%
12\\-5\\1
\end{tabular}\endgroup%
\kern3pt%
\begingroup \smaller\smaller\smaller\begin{tabular}{@{}c@{}}%
3\\-1\\1
\end{tabular}\endgroup%
\kern3pt%
\begingroup \smaller\smaller\smaller\begin{tabular}{@{}c@{}}%
2\\0\\1
\end{tabular}\endgroup%
{$\left.\llap{\phantom{%
\begingroup \smaller\smaller\smaller\begin{tabular}{@{}c@{}}%
0\\0\\0
\end{tabular}\endgroup%
}}\!\right]$}%
}%
\ifdim\wd\matricesbox>\halfwidth\myboxwidth=\hsize\else\myboxwidth=\halfwidth\fi
\vbox{%
\ifdim\myboxwidth=\hsize
\setbox\onelinebox=\hbox{%
\vbox{\hbox{%
$\Pi_{10,10}$ spans $L_{123.8}$%
}\hbox{%
$2\slashtwo22|22\slashtwo22|2\rtimes D_{4}$%
}%
}%
\hfill\copy\matricesbox
}%
\ifdim\wd\onelinebox>\myboxwidth
\hbox to \myboxwidth{%
$\Pi_{10,10}$ spans $L_{123.8}$%
\hfil
$2\slashtwo22|22\slashtwo22|2\rtimes D_{4}$%
}%
\box\matricesbox
\else
\hbox to \myboxwidth{%
\unhbox\onelinebox
}%
\fi
\else
\hbox to \myboxwidth{%
$\Pi_{10,10}$ spans $L_{123.8}$%
\hfil}%
\hbox to \myboxwidth{%
$2\slashtwo22|22\slashtwo22|2\rtimes D_{4}$%
\hfil}%
\box\matricesbox
\fi
}%
\hfill\discretionary{}{}{}%
\setbox\matricesbox=\hbox{%
{$\left[\!\llap{\phantom{%
\begingroup \smaller\smaller\smaller\begin{tabular}{@{}c@{}}%
\phantom{0}\\\phantom{0}\\\phantom{0}
\end{tabular}\endgroup%
}}\right.$}%
\begingroup \smaller\smaller\smaller\begin{tabular}{@{}c@{}}%
-3\\\phantom{0}\\\phantom{0}
\end{tabular}\endgroup%
\kern3pt%
\begingroup \smaller\smaller\smaller\begin{tabular}{@{}c@{}}%
\phantom{0}\\1\\\phantom{0}
\end{tabular}\endgroup%
\kern3pt%
\begingroup \smaller\smaller\smaller\begin{tabular}{@{}c@{}}%
\phantom{0}\\\phantom{0}\\3
\end{tabular}\endgroup%
{$\left.\llap{\phantom{%
\begingroup \smaller\smaller\smaller\begin{tabular}{@{}c@{}}%
\phantom{0}\\\phantom{0}\\\phantom{0}
\end{tabular}\endgroup%
}}\!\right]$}%
{$\left[\!\llap{\phantom{%
\begingroup \smaller\smaller\smaller\begin{tabular}{@{}c@{}}%
0\\0\\0
\end{tabular}\endgroup%
}}\right.$}%
\begingroup \smaller\smaller\smaller\begin{tabular}{@{}c@{}}%
1\\-2\\0
\end{tabular}\endgroup%
\kern3pt%
\begingroup \smaller\smaller\smaller\begin{tabular}{@{}c@{}}%
4\\-5\\3
\end{tabular}\endgroup%
\kern3pt%
\begingroup \smaller\smaller\smaller\begin{tabular}{@{}c@{}}%
4\\-2\\4
\end{tabular}\endgroup%
{$\left.\llap{\phantom{%
\begingroup \smaller\smaller\smaller\begin{tabular}{@{}c@{}}%
0\\0\\0
\end{tabular}\endgroup%
}}\!\right]$}%
}%
\ifdim\wd\matricesbox>\halfwidth\myboxwidth=\hsize\else\myboxwidth=\halfwidth\fi
\vbox{%
\ifdim\myboxwidth=\hsize
\setbox\onelinebox=\hbox{%
\vbox{\hbox{%
$\Pi_{10,11}=\hbox{GN}_{56}$ spans $L_{7.7}$%
}\hbox{%
$3\infty|\infty3\slashinfty3\infty|\infty3\slashinfty\rtimes D_{4}$%
}%
}%
\hfill\copy\matricesbox
}%
\ifdim\wd\onelinebox>\myboxwidth
\hbox to \myboxwidth{%
$\Pi_{10,11}=\hbox{GN}_{56}$ spans $L_{7.7}$%
\hfil
$3\infty|\infty3\slashinfty3\infty|\infty3\slashinfty\rtimes D_{4}$%
}%
\box\matricesbox
\else
\hbox to \myboxwidth{%
\unhbox\onelinebox
}%
\fi
\else
\hbox to \myboxwidth{%
$\Pi_{10,11}=\hbox{GN}_{56}$ spans $L_{7.7}$%
\hfil}%
\hbox to \myboxwidth{%
$3\infty|\infty3\slashinfty3\infty|\infty3\slashinfty\rtimes D_{4}$%
\hfil}%
\box\matricesbox
\fi
}%
\hfill\discretionary{}{}{}%
\setbox\matricesbox=\hbox{%
{$\left[\!\llap{\phantom{%
\begingroup \smaller\smaller\smaller\begin{tabular}{@{}c@{}}%
\phantom{0}\\\phantom{0}\\\phantom{0}
\end{tabular}\endgroup%
}}\right.$}%
\begingroup \smaller\smaller\smaller\begin{tabular}{@{}c@{}}%
-4\\\phantom{0}\\\phantom{0}
\end{tabular}\endgroup%
\kern3pt%
\begingroup \smaller\smaller\smaller\begin{tabular}{@{}c@{}}%
\phantom{0}\\1/2\\\phantom{0}
\end{tabular}\endgroup%
\kern3pt%
\begingroup \smaller\smaller\smaller\begin{tabular}{@{}c@{}}%
\phantom{0}\\\phantom{0}\\9/2
\end{tabular}\endgroup%
{$\left.\llap{\phantom{%
\begingroup \smaller\smaller\smaller\begin{tabular}{@{}c@{}}%
\phantom{0}\\\phantom{0}\\\phantom{0}
\end{tabular}\endgroup%
}}\!\right]$}%
{$\left[\!\llap{\phantom{%
\begingroup \smaller\smaller\smaller\begin{tabular}{@{}c@{}}%
0\\0\\0
\end{tabular}\endgroup%
}}\right.$}%
\begingroup \smaller\smaller\smaller\begin{tabular}{@{}c@{}}%
2\\6\\0
\end{tabular}\endgroup%
\kern3pt%
\begingroup \smaller\smaller\smaller\begin{tabular}{@{}c@{}}%
4\\10\\2
\end{tabular}\endgroup%
\kern3pt%
\begingroup \smaller\smaller\smaller\begin{tabular}{@{}c@{}}%
1\\1\\1
\end{tabular}\endgroup%
{$\left.\llap{\phantom{%
\begingroup \smaller\smaller\smaller\begin{tabular}{@{}c@{}}%
0\\0\\0
\end{tabular}\endgroup%
}}\!\right]$}%
}%
\ifdim\wd\matricesbox>\halfwidth\myboxwidth=\hsize\else\myboxwidth=\halfwidth\fi
\vbox{%
\ifdim\myboxwidth=\hsize
\setbox\onelinebox=\hbox{%
\vbox{\hbox{%
$\Pi_{10,12}$ spans $L_{216.1}$%
}\hbox{%
$|4\infty\slashtwo\infty4|4\infty\slashtwo\infty4\rtimes D_{4}$%
}%
}%
\hfill\copy\matricesbox
}%
\ifdim\wd\onelinebox>\myboxwidth
\hbox to \myboxwidth{%
$\Pi_{10,12}$ spans $L_{216.1}$%
\hfil
$|4\infty\slashtwo\infty4|4\infty\slashtwo\infty4\rtimes D_{4}$%
}%
\box\matricesbox
\else
\hbox to \myboxwidth{%
\unhbox\onelinebox
}%
\fi
\else
\hbox to \myboxwidth{%
$\Pi_{10,12}$ spans $L_{216.1}$%
\hfil}%
\hbox to \myboxwidth{%
$|4\infty\slashtwo\infty4|4\infty\slashtwo\infty4\rtimes D_{4}$%
\hfil}%
\box\matricesbox
\fi
}%
\hfill\discretionary{}{}{}%
\setbox\matricesbox=\hbox{%
{$\left[\!\llap{\phantom{%
\begingroup \smaller\smaller\smaller\begin{tabular}{@{}c@{}}%
\phantom{0}\\\phantom{0}\\\phantom{0}
\end{tabular}\endgroup%
}}\right.$}%
\begingroup \smaller\smaller\smaller\begin{tabular}{@{}c@{}}%
-1\\\phantom{0}\\\phantom{0}
\end{tabular}\endgroup%
\kern3pt%
\begingroup \smaller\smaller\smaller\begin{tabular}{@{}c@{}}%
\phantom{0}\\5\\\phantom{0}
\end{tabular}\endgroup%
\kern3pt%
\begingroup \smaller\smaller\smaller\begin{tabular}{@{}c@{}}%
\phantom{0}\\\phantom{0}\\25
\end{tabular}\endgroup%
{$\left.\llap{\phantom{%
\begingroup \smaller\smaller\smaller\begin{tabular}{@{}c@{}}%
\phantom{0}\\\phantom{0}\\\phantom{0}
\end{tabular}\endgroup%
}}\!\right]$}%
{$\left[\!\llap{\phantom{%
\begingroup \smaller\smaller\smaller\begin{tabular}{@{}c@{}}%
0\\0\\0
\end{tabular}\endgroup%
}}\right.$}%
\begingroup \smaller\smaller\smaller\begin{tabular}{@{}c@{}}%
4\\-2\\0
\end{tabular}\endgroup%
\kern3pt%
\begingroup \smaller\smaller\smaller\begin{tabular}{@{}c@{}}%
20\\-8\\2
\end{tabular}\endgroup%
\kern3pt%
\begingroup \smaller\smaller\smaller\begin{tabular}{@{}c@{}}%
5\\-1\\1
\end{tabular}\endgroup%
{$\left.\llap{\phantom{%
\begingroup \smaller\smaller\smaller\begin{tabular}{@{}c@{}}%
0\\0\\0
\end{tabular}\endgroup%
}}\!\right]$}%
}%
\ifdim\wd\matricesbox>\halfwidth\myboxwidth=\hsize\else\myboxwidth=\halfwidth\fi
\vbox{%
\ifdim\myboxwidth=\hsize
\setbox\onelinebox=\hbox{%
\vbox{\hbox{%
$\Pi_{10,13}$ spans $L_{149.20}$%
}\hbox{%
$\infty2|2\infty\slashinfty\infty2|2\infty\slashinfty\rtimes D_{4}$%
}%
}%
\hfill\copy\matricesbox
}%
\ifdim\wd\onelinebox>\myboxwidth
\hbox to \myboxwidth{%
$\Pi_{10,13}$ spans $L_{149.20}$%
\hfil
$\infty2|2\infty\slashinfty\infty2|2\infty\slashinfty\rtimes D_{4}$%
}%
\box\matricesbox
\else
\hbox to \myboxwidth{%
\unhbox\onelinebox
}%
\fi
\else
\hbox to \myboxwidth{%
$\Pi_{10,13}$ spans $L_{149.20}$%
\hfil}%
\hbox to \myboxwidth{%
$\infty2|2\infty\slashinfty\infty2|2\infty\slashinfty\rtimes D_{4}$%
\hfil}%
\box\matricesbox
\fi
}%
\hfill\discretionary{}{}{}%
\setbox\matricesbox=\hbox{%
{$\left[\!\llap{\phantom{%
\begingroup \smaller\smaller\smaller\begin{tabular}{@{}c@{}}%
\phantom{0}\\\phantom{0}\\\phantom{0}
\end{tabular}\endgroup%
}}\right.$}%
\begingroup \smaller\smaller\smaller\begin{tabular}{@{}c@{}}%
-1/4\\\phantom{0}\\\phantom{0}
\end{tabular}\endgroup%
\kern3pt%
\begingroup \smaller\smaller\smaller\begin{tabular}{@{}c@{}}%
\phantom{0}\\35/4\\\phantom{0}
\end{tabular}\endgroup%
\kern3pt%
\begingroup \smaller\smaller\smaller\begin{tabular}{@{}c@{}}%
\phantom{0}\\\phantom{0}\\21/2
\end{tabular}\endgroup%
{$\left.\llap{\phantom{%
\begingroup \smaller\smaller\smaller\begin{tabular}{@{}c@{}}%
\phantom{0}\\\phantom{0}\\\phantom{0}
\end{tabular}\endgroup%
}}\!\right]$}%
{$\left[\!\llap{\phantom{%
\begingroup \smaller\smaller\smaller\begin{tabular}{@{}c@{}}%
0\\0\\0
\end{tabular}\endgroup%
}}\right.$}%
\begingroup \smaller\smaller\smaller\begin{tabular}{@{}c@{}}%
10\\2\\0
\end{tabular}\endgroup%
\kern3pt%
\begingroup \smaller\smaller\smaller\begin{tabular}{@{}c@{}}%
42\\6\\4
\end{tabular}\endgroup%
\kern3pt%
\begingroup \smaller\smaller\smaller\begin{tabular}{@{}c@{}}%
35\\3\\5
\end{tabular}\endgroup%
\kern3pt%
\begingroup \smaller\smaller\smaller\begin{tabular}{@{}c@{}}%
24\\0\\4
\end{tabular}\endgroup%
\kern3pt%
\begingroup \smaller\smaller\smaller\begin{tabular}{@{}c@{}}%
7\\-1\\1
\end{tabular}\endgroup%
\kern3pt%
\begingroup \smaller\smaller\smaller\begin{tabular}{@{}c@{}}%
10\\-2\\0
\end{tabular}\endgroup%
{$\left.\llap{\phantom{%
\begingroup \smaller\smaller\smaller\begin{tabular}{@{}c@{}}%
0\\0\\0
\end{tabular}\endgroup%
}}\!\right]$}%
}%
\ifdim\wd\matricesbox>\halfwidth\myboxwidth=\hsize\else\myboxwidth=\halfwidth\fi
\vbox{%
\ifdim\myboxwidth=\hsize
\setbox\onelinebox=\hbox{%
\vbox{\hbox{%
$\Pi_{10,14}$ spans $L_{136.10}$%
}\hbox{%
$|22222|22222\rtimes D_{2}$%
}%
}%
\hfill\copy\matricesbox
}%
\ifdim\wd\onelinebox>\myboxwidth
\hbox to \myboxwidth{%
$\Pi_{10,14}$ spans $L_{136.10}$%
\hfil
$|22222|22222\rtimes D_{2}$%
}%
\box\matricesbox
\else
\hbox to \myboxwidth{%
\unhbox\onelinebox
}%
\fi
\else
\hbox to \myboxwidth{%
$\Pi_{10,14}$ spans $L_{136.10}$%
\hfil}%
\hbox to \myboxwidth{%
$|22222|22222\rtimes D_{2}$%
\hfil}%
\box\matricesbox
\fi
}%
\hfill\discretionary{}{}{}%
\setbox\matricesbox=\hbox{%
{$\left[\!\llap{\phantom{%
\begingroup \smaller\smaller\smaller\begin{tabular}{@{}c@{}}%
\phantom{0}\\\phantom{0}\\\phantom{0}
\end{tabular}\endgroup%
}}\right.$}%
\begingroup \smaller\smaller\smaller\begin{tabular}{@{}c@{}}%
-1/4\\\phantom{0}\\\phantom{0}
\end{tabular}\endgroup%
\kern3pt%
\begingroup \smaller\smaller\smaller\begin{tabular}{@{}c@{}}%
\phantom{0}\\15/2\\\phantom{0}
\end{tabular}\endgroup%
\kern3pt%
\begingroup \smaller\smaller\smaller\begin{tabular}{@{}c@{}}%
\phantom{0}\\\phantom{0}\\15/2
\end{tabular}\endgroup%
{$\left.\llap{\phantom{%
\begingroup \smaller\smaller\smaller\begin{tabular}{@{}c@{}}%
\phantom{0}\\\phantom{0}\\\phantom{0}
\end{tabular}\endgroup%
}}\!\right]$}%
{$\left[\!\llap{\phantom{%
\begingroup \smaller\smaller\smaller\begin{tabular}{@{}c@{}}%
0\\0\\0
\end{tabular}\endgroup%
}}\right.$}%
\begingroup \smaller\smaller\smaller\begin{tabular}{@{}c@{}}%
20\\4\\0
\end{tabular}\endgroup%
\kern3pt%
\begingroup \smaller\smaller\smaller\begin{tabular}{@{}c@{}}%
30\\5\\-3
\end{tabular}\endgroup%
\kern3pt%
\begingroup \smaller\smaller\smaller\begin{tabular}{@{}c@{}}%
30\\3\\-5
\end{tabular}\endgroup%
\kern3pt%
\begingroup \smaller\smaller\smaller\begin{tabular}{@{}c@{}}%
20\\0\\-4
\end{tabular}\endgroup%
\kern3pt%
\begingroup \smaller\smaller\smaller\begin{tabular}{@{}c@{}}%
6\\-1\\-1
\end{tabular}\endgroup%
\kern3pt%
\begingroup \smaller\smaller\smaller\begin{tabular}{@{}c@{}}%
20\\-4\\0
\end{tabular}\endgroup%
{$\left.\llap{\phantom{%
\begingroup \smaller\smaller\smaller\begin{tabular}{@{}c@{}}%
0\\0\\0
\end{tabular}\endgroup%
}}\!\right]$}%
}%
\ifdim\wd\matricesbox>\halfwidth\myboxwidth=\hsize\else\myboxwidth=\halfwidth\fi
\vbox{%
\ifdim\myboxwidth=\hsize
\setbox\onelinebox=\hbox{%
\vbox{\hbox{%
$\Pi_{10,15}$ spans $L_{19.10}$%
}\hbox{%
$2222|22222|2\rtimes D_{2}$%
}%
}%
\hfill\copy\matricesbox
}%
\ifdim\wd\onelinebox>\myboxwidth
\hbox to \myboxwidth{%
$\Pi_{10,15}$ spans $L_{19.10}$%
\hfil
$2222|22222|2\rtimes D_{2}$%
}%
\box\matricesbox
\else
\hbox to \myboxwidth{%
\unhbox\onelinebox
}%
\fi
\else
\hbox to \myboxwidth{%
$\Pi_{10,15}$ spans $L_{19.10}$%
\hfil}%
\hbox to \myboxwidth{%
$2222|22222|2\rtimes D_{2}$%
\hfil}%
\box\matricesbox
\fi
}%
\hfill\discretionary{}{}{}%
\setbox\matricesbox=\hbox{%
{$\left[\!\llap{\phantom{%
\begingroup \smaller\smaller\smaller\begin{tabular}{@{}c@{}}%
\phantom{0}\\\phantom{0}\\\phantom{0}
\end{tabular}\endgroup%
}}\right.$}%
\begingroup \smaller\smaller\smaller\begin{tabular}{@{}c@{}}%
-1/2\\\phantom{0}\\\phantom{0}
\end{tabular}\endgroup%
\kern3pt%
\begingroup \smaller\smaller\smaller\begin{tabular}{@{}c@{}}%
\phantom{0}\\10\\\phantom{0}
\end{tabular}\endgroup%
\kern3pt%
\begingroup \smaller\smaller\smaller\begin{tabular}{@{}c@{}}%
\phantom{0}\\\phantom{0}\\15/2
\end{tabular}\endgroup%
{$\left.\llap{\phantom{%
\begingroup \smaller\smaller\smaller\begin{tabular}{@{}c@{}}%
\phantom{0}\\\phantom{0}\\\phantom{0}
\end{tabular}\endgroup%
}}\!\right]$}%
{$\left[\!\llap{\phantom{%
\begingroup \smaller\smaller\smaller\begin{tabular}{@{}c@{}}%
0\\0\\0
\end{tabular}\endgroup%
}}\right.$}%
\begingroup \smaller\smaller\smaller\begin{tabular}{@{}c@{}}%
8\\2\\0
\end{tabular}\endgroup%
\kern3pt%
\begingroup \smaller\smaller\smaller\begin{tabular}{@{}c@{}}%
5\\1\\-1
\end{tabular}\endgroup%
\kern3pt%
\begingroup \smaller\smaller\smaller\begin{tabular}{@{}c@{}}%
3\\0\\-1
\end{tabular}\endgroup%
\kern3pt%
\begingroup \smaller\smaller\smaller\begin{tabular}{@{}c@{}}%
40\\-6\\-8
\end{tabular}\endgroup%
\kern3pt%
\begingroup \smaller\smaller\smaller\begin{tabular}{@{}c@{}}%
30\\-6\\-4
\end{tabular}\endgroup%
\kern3pt%
\begingroup \smaller\smaller\smaller\begin{tabular}{@{}c@{}}%
8\\-2\\0
\end{tabular}\endgroup%
{$\left.\llap{\phantom{%
\begingroup \smaller\smaller\smaller\begin{tabular}{@{}c@{}}%
0\\0\\0
\end{tabular}\endgroup%
}}\!\right]$}%
}%
\ifdim\wd\matricesbox>\halfwidth\myboxwidth=\hsize\else\myboxwidth=\halfwidth\fi
\vbox{%
\ifdim\myboxwidth=\hsize
\setbox\onelinebox=\hbox{%
\vbox{\hbox{%
$\Pi_{10,16}$ spans $L_{17.11}$%
}\hbox{%
$22|22222|222\rtimes D_{2}$%
}%
}%
\hfill\copy\matricesbox
}%
\ifdim\wd\onelinebox>\myboxwidth
\hbox to \myboxwidth{%
$\Pi_{10,16}$ spans $L_{17.11}$%
\hfil
$22|22222|222\rtimes D_{2}$%
}%
\box\matricesbox
\else
\hbox to \myboxwidth{%
\unhbox\onelinebox
}%
\fi
\else
\hbox to \myboxwidth{%
$\Pi_{10,16}$ spans $L_{17.11}$%
\hfil}%
\hbox to \myboxwidth{%
$22|22222|222\rtimes D_{2}$%
\hfil}%
\box\matricesbox
\fi
}%
\hfill\discretionary{}{}{}%
\setbox\matricesbox=\hbox{%
{$\left[\!\llap{\phantom{%
\begingroup \smaller\smaller\smaller\begin{tabular}{@{}c@{}}%
\phantom{0}\\\phantom{0}\\\phantom{0}
\end{tabular}\endgroup%
}}\right.$}%
\begingroup \smaller\smaller\smaller\begin{tabular}{@{}c@{}}%
-1/2\\\phantom{0}\\\phantom{0}
\end{tabular}\endgroup%
\kern3pt%
\begingroup \smaller\smaller\smaller\begin{tabular}{@{}c@{}}%
\phantom{0}\\15/2\\\phantom{0}
\end{tabular}\endgroup%
\kern3pt%
\begingroup \smaller\smaller\smaller\begin{tabular}{@{}c@{}}%
\phantom{0}\\\phantom{0}\\10
\end{tabular}\endgroup%
{$\left.\llap{\phantom{%
\begingroup \smaller\smaller\smaller\begin{tabular}{@{}c@{}}%
\phantom{0}\\\phantom{0}\\\phantom{0}
\end{tabular}\endgroup%
}}\!\right]$}%
{$\left[\!\llap{\phantom{%
\begingroup \smaller\smaller\smaller\begin{tabular}{@{}c@{}}%
0\\0\\0
\end{tabular}\endgroup%
}}\right.$}%
\begingroup \smaller\smaller\smaller\begin{tabular}{@{}c@{}}%
3\\-1\\0
\end{tabular}\endgroup%
\kern3pt%
\begingroup \smaller\smaller\smaller\begin{tabular}{@{}c@{}}%
5\\-1\\-1
\end{tabular}\endgroup%
\kern3pt%
\begingroup \smaller\smaller\smaller\begin{tabular}{@{}c@{}}%
8\\0\\-2
\end{tabular}\endgroup%
\kern3pt%
\begingroup \smaller\smaller\smaller\begin{tabular}{@{}c@{}}%
30\\4\\-6
\end{tabular}\endgroup%
\kern3pt%
\begingroup \smaller\smaller\smaller\begin{tabular}{@{}c@{}}%
40\\8\\-6
\end{tabular}\endgroup%
\kern3pt%
\begingroup \smaller\smaller\smaller\begin{tabular}{@{}c@{}}%
3\\1\\0
\end{tabular}\endgroup%
{$\left.\llap{\phantom{%
\begingroup \smaller\smaller\smaller\begin{tabular}{@{}c@{}}%
0\\0\\0
\end{tabular}\endgroup%
}}\!\right]$}%
}%
\ifdim\wd\matricesbox>\halfwidth\myboxwidth=\hsize\else\myboxwidth=\halfwidth\fi
\vbox{%
\ifdim\myboxwidth=\hsize
\setbox\onelinebox=\hbox{%
\vbox{\hbox{%
$\Pi_{10,17}$ spans $L_{17.11}$%
}\hbox{%
$|22222|22222\rtimes D_{2}$%
}%
}%
\hfill\copy\matricesbox
}%
\ifdim\wd\onelinebox>\myboxwidth
\hbox to \myboxwidth{%
$\Pi_{10,17}$ spans $L_{17.11}$%
\hfil
$|22222|22222\rtimes D_{2}$%
}%
\box\matricesbox
\else
\hbox to \myboxwidth{%
\unhbox\onelinebox
}%
\fi
\else
\hbox to \myboxwidth{%
$\Pi_{10,17}$ spans $L_{17.11}$%
\hfil}%
\hbox to \myboxwidth{%
$|22222|22222\rtimes D_{2}$%
\hfil}%
\box\matricesbox
\fi
}%
\hfill\discretionary{}{}{}%
\setbox\matricesbox=\hbox{%
{$\left[\!\llap{\phantom{%
\begingroup \smaller\smaller\smaller\begin{tabular}{@{}c@{}}%
\phantom{0}\\\phantom{0}\\\phantom{0}
\end{tabular}\endgroup%
}}\right.$}%
\begingroup \smaller\smaller\smaller\begin{tabular}{@{}c@{}}%
-1\\\phantom{0}\\\phantom{0}
\end{tabular}\endgroup%
\kern3pt%
\begingroup \smaller\smaller\smaller\begin{tabular}{@{}c@{}}%
\phantom{0}\\6\\\phantom{0}
\end{tabular}\endgroup%
\kern3pt%
\begingroup \smaller\smaller\smaller\begin{tabular}{@{}c@{}}%
\phantom{0}\\\phantom{0}\\6
\end{tabular}\endgroup%
{$\left.\llap{\phantom{%
\begingroup \smaller\smaller\smaller\begin{tabular}{@{}c@{}}%
\phantom{0}\\\phantom{0}\\\phantom{0}
\end{tabular}\endgroup%
}}\!\right]$}%
{$\left[\!\llap{\phantom{%
\begingroup \smaller\smaller\smaller\begin{tabular}{@{}c@{}}%
0\\0\\0
\end{tabular}\endgroup%
}}\right.$}%
\begingroup \smaller\smaller\smaller\begin{tabular}{@{}c@{}}%
2\\1\\0
\end{tabular}\endgroup%
\kern3pt%
\begingroup \smaller\smaller\smaller\begin{tabular}{@{}c@{}}%
3\\1\\1
\end{tabular}\endgroup%
\kern3pt%
\begingroup \smaller\smaller\smaller\begin{tabular}{@{}c@{}}%
2\\0\\1
\end{tabular}\endgroup%
\kern3pt%
\begingroup \smaller\smaller\smaller\begin{tabular}{@{}c@{}}%
24\\-6\\8
\end{tabular}\endgroup%
\kern3pt%
\begingroup \smaller\smaller\smaller\begin{tabular}{@{}c@{}}%
24\\-8\\6
\end{tabular}\endgroup%
\kern3pt%
\begingroup \smaller\smaller\smaller\begin{tabular}{@{}c@{}}%
2\\-1\\0
\end{tabular}\endgroup%
{$\left.\llap{\phantom{%
\begingroup \smaller\smaller\smaller\begin{tabular}{@{}c@{}}%
0\\0\\0
\end{tabular}\endgroup%
}}\!\right]$}%
}%
\ifdim\wd\matricesbox>\halfwidth\myboxwidth=\hsize\else\myboxwidth=\halfwidth\fi
\vbox{%
\ifdim\myboxwidth=\hsize
\setbox\onelinebox=\hbox{%
\vbox{\hbox{%
$\Pi_{10,18}$ spans $L_{123.8}$%
}\hbox{%
$22|22222|222\rtimes D_{2}$%
}%
}%
\hfill\copy\matricesbox
}%
\ifdim\wd\onelinebox>\myboxwidth
\hbox to \myboxwidth{%
$\Pi_{10,18}$ spans $L_{123.8}$%
\hfil
$22|22222|222\rtimes D_{2}$%
}%
\box\matricesbox
\else
\hbox to \myboxwidth{%
\unhbox\onelinebox
}%
\fi
\else
\hbox to \myboxwidth{%
$\Pi_{10,18}$ spans $L_{123.8}$%
\hfil}%
\hbox to \myboxwidth{%
$22|22222|222\rtimes D_{2}$%
\hfil}%
\box\matricesbox
\fi
}%
\hfill\discretionary{}{}{}%
\setbox\matricesbox=\hbox{%
{$\left[\!\llap{\phantom{%
\begingroup \smaller\smaller\smaller\begin{tabular}{@{}c@{}}%
\phantom{0}\\\phantom{0}\\\phantom{0}
\end{tabular}\endgroup%
}}\right.$}%
\begingroup \smaller\smaller\smaller\begin{tabular}{@{}c@{}}%
-1/8\\\phantom{0}\\\phantom{0}
\end{tabular}\endgroup%
\kern3pt%
\begingroup \smaller\smaller\smaller\begin{tabular}{@{}c@{}}%
\phantom{0}\\21\\\phantom{0}
\end{tabular}\endgroup%
\kern3pt%
\begingroup \smaller\smaller\smaller\begin{tabular}{@{}c@{}}%
\phantom{0}\\\phantom{0}\\35/2
\end{tabular}\endgroup%
{$\left.\llap{\phantom{%
\begingroup \smaller\smaller\smaller\begin{tabular}{@{}c@{}}%
\phantom{0}\\\phantom{0}\\\phantom{0}
\end{tabular}\endgroup%
}}\!\right]$}%
{$\left[\!\llap{\phantom{%
\begingroup \smaller\smaller\smaller\begin{tabular}{@{}c@{}}%
0\\0\\0
\end{tabular}\endgroup%
}}\right.$}%
\begingroup \smaller\smaller\smaller\begin{tabular}{@{}c@{}}%
48\\4\\0
\end{tabular}\endgroup%
\kern3pt%
\begingroup \smaller\smaller\smaller\begin{tabular}{@{}c@{}}%
70\\5\\-3
\end{tabular}\endgroup%
\kern3pt%
\begingroup \smaller\smaller\smaller\begin{tabular}{@{}c@{}}%
84\\4\\-6
\end{tabular}\endgroup%
\kern3pt%
\begingroup \smaller\smaller\smaller\begin{tabular}{@{}c@{}}%
20\\0\\-2
\end{tabular}\endgroup%
\kern3pt%
\begingroup \smaller\smaller\smaller\begin{tabular}{@{}c@{}}%
14\\-1\\-1
\end{tabular}\endgroup%
\kern3pt%
\begingroup \smaller\smaller\smaller\begin{tabular}{@{}c@{}}%
48\\-4\\0
\end{tabular}\endgroup%
{$\left.\llap{\phantom{%
\begingroup \smaller\smaller\smaller\begin{tabular}{@{}c@{}}%
0\\0\\0
\end{tabular}\endgroup%
}}\!\right]$}%
}%
\ifdim\wd\matricesbox>\halfwidth\myboxwidth=\hsize\else\myboxwidth=\halfwidth\fi
\vbox{%
\ifdim\myboxwidth=\hsize
\setbox\onelinebox=\hbox{%
\vbox{\hbox{%
$\Pi_{10,19}$ spans $L_{69.11}$%
}\hbox{%
$222|22222|22\rtimes D_{2}$%
}%
}%
\hfill\copy\matricesbox
}%
\ifdim\wd\onelinebox>\myboxwidth
\hbox to \myboxwidth{%
$\Pi_{10,19}$ spans $L_{69.11}$%
\hfil
$222|22222|22\rtimes D_{2}$%
}%
\box\matricesbox
\else
\hbox to \myboxwidth{%
\unhbox\onelinebox
}%
\fi
\else
\hbox to \myboxwidth{%
$\Pi_{10,19}$ spans $L_{69.11}$%
\hfil}%
\hbox to \myboxwidth{%
$222|22222|22\rtimes D_{2}$%
\hfil}%
\box\matricesbox
\fi
}%
\hfill\discretionary{}{}{}%
\setbox\matricesbox=\hbox{%
{$\left[\!\llap{\phantom{%
\begingroup \smaller\smaller\smaller\begin{tabular}{@{}c@{}}%
\phantom{0}\\\phantom{0}\\\phantom{0}
\end{tabular}\endgroup%
}}\right.$}%
\begingroup \smaller\smaller\smaller\begin{tabular}{@{}c@{}}%
-1\\\phantom{0}\\\phantom{0}
\end{tabular}\endgroup%
\kern3pt%
\begingroup \smaller\smaller\smaller\begin{tabular}{@{}c@{}}%
\phantom{0}\\3\\\phantom{0}
\end{tabular}\endgroup%
\kern3pt%
\begingroup \smaller\smaller\smaller\begin{tabular}{@{}c@{}}%
\phantom{0}\\\phantom{0}\\3
\end{tabular}\endgroup%
{$\left.\llap{\phantom{%
\begingroup \smaller\smaller\smaller\begin{tabular}{@{}c@{}}%
\phantom{0}\\\phantom{0}\\\phantom{0}
\end{tabular}\endgroup%
}}\!\right]$}%
{$\left[\!\llap{\phantom{%
\begingroup \smaller\smaller\smaller\begin{tabular}{@{}c@{}}%
0\\0\\0
\end{tabular}\endgroup%
}}\right.$}%
\begingroup \smaller\smaller\smaller\begin{tabular}{@{}c@{}}%
3\\-2\\0
\end{tabular}\endgroup%
\kern3pt%
\begingroup \smaller\smaller\smaller\begin{tabular}{@{}c@{}}%
12\\-6\\4
\end{tabular}\endgroup%
\kern3pt%
\begingroup \smaller\smaller\smaller\begin{tabular}{@{}c@{}}%
12\\-4\\6
\end{tabular}\endgroup%
\kern3pt%
\begingroup \smaller\smaller\smaller\begin{tabular}{@{}c@{}}%
3\\0\\2
\end{tabular}\endgroup%
\kern3pt%
\begingroup \smaller\smaller\smaller\begin{tabular}{@{}c@{}}%
2\\1\\1
\end{tabular}\endgroup%
\kern3pt%
\begingroup \smaller\smaller\smaller\begin{tabular}{@{}c@{}}%
3\\2\\0
\end{tabular}\endgroup%
{$\left.\llap{\phantom{%
\begingroup \smaller\smaller\smaller\begin{tabular}{@{}c@{}}%
0\\0\\0
\end{tabular}\endgroup%
}}\!\right]$}%
}%
\ifdim\wd\matricesbox>\halfwidth\myboxwidth=\hsize\else\myboxwidth=\halfwidth\fi
\vbox{%
\ifdim\myboxwidth=\hsize
\setbox\onelinebox=\hbox{%
\vbox{\hbox{%
$\Pi_{10,20}$ spans $L_{123.8}$%
}\hbox{%
$|22222|22222\rtimes D_{2}$%
}%
}%
\hfill\copy\matricesbox
}%
\ifdim\wd\onelinebox>\myboxwidth
\hbox to \myboxwidth{%
$\Pi_{10,20}$ spans $L_{123.8}$%
\hfil
$|22222|22222\rtimes D_{2}$%
}%
\box\matricesbox
\else
\hbox to \myboxwidth{%
\unhbox\onelinebox
}%
\fi
\else
\hbox to \myboxwidth{%
$\Pi_{10,20}$ spans $L_{123.8}$%
\hfil}%
\hbox to \myboxwidth{%
$|22222|22222\rtimes D_{2}$%
\hfil}%
\box\matricesbox
\fi
}%
\hfill\discretionary{}{}{}%
\setbox\matricesbox=\hbox{%
{$\left[\!\llap{\phantom{%
\begingroup \smaller\smaller\smaller\begin{tabular}{@{}c@{}}%
\phantom{0}\\\phantom{0}\\\phantom{0}
\end{tabular}\endgroup%
}}\right.$}%
\begingroup \smaller\smaller\smaller\begin{tabular}{@{}c@{}}%
-3\\\phantom{0}\\\phantom{0}
\end{tabular}\endgroup%
\kern3pt%
\begingroup \smaller\smaller\smaller\begin{tabular}{@{}c@{}}%
\phantom{0}\\3\\\phantom{0}
\end{tabular}\endgroup%
\kern3pt%
\begingroup \smaller\smaller\smaller\begin{tabular}{@{}c@{}}%
\phantom{0}\\\phantom{0}\\1
\end{tabular}\endgroup%
{$\left.\llap{\phantom{%
\begingroup \smaller\smaller\smaller\begin{tabular}{@{}c@{}}%
\phantom{0}\\\phantom{0}\\\phantom{0}
\end{tabular}\endgroup%
}}\!\right]$}%
{$\left[\!\llap{\phantom{%
\begingroup \smaller\smaller\smaller\begin{tabular}{@{}c@{}}%
0\\0\\0
\end{tabular}\endgroup%
}}\right.$}%
\begingroup \smaller\smaller\smaller\begin{tabular}{@{}c@{}}%
1\\1\\-1
\end{tabular}\endgroup%
\kern3pt%
\begingroup \smaller\smaller\smaller\begin{tabular}{@{}c@{}}%
4\\1\\-7
\end{tabular}\endgroup%
\kern3pt%
\begingroup \smaller\smaller\smaller\begin{tabular}{@{}c@{}}%
4\\-1\\-7
\end{tabular}\endgroup%
\kern3pt%
\begingroup \smaller\smaller\smaller\begin{tabular}{@{}c@{}}%
4\\-3\\-5
\end{tabular}\endgroup%
\kern3pt%
\begingroup \smaller\smaller\smaller\begin{tabular}{@{}c@{}}%
4\\-4\\-2
\end{tabular}\endgroup%
{$\left.\llap{\phantom{%
\begingroup \smaller\smaller\smaller\begin{tabular}{@{}c@{}}%
0\\0\\0
\end{tabular}\endgroup%
}}\!\right]$}%
}%
\ifdim\wd\matricesbox>\halfwidth\myboxwidth=\hsize\else\myboxwidth=\halfwidth\fi
\vbox{%
\ifdim\myboxwidth=\hsize
\setbox\onelinebox=\hbox{%
\vbox{\hbox{%
$\Pi_{10,21}=\hbox{GN}_{54}$ spans $L_{7.7}$%
}\hbox{%
$3\infty\slashinfty\infty3\infty3\slashinfty3\infty\rtimes D_{2}$%
}%
}%
\hfill\copy\matricesbox
}%
\ifdim\wd\onelinebox>\myboxwidth
\hbox to \myboxwidth{%
$\Pi_{10,21}=\hbox{GN}_{54}$ spans $L_{7.7}$%
\hfil
$3\infty\slashinfty\infty3\infty3\slashinfty3\infty\rtimes D_{2}$%
}%
\box\matricesbox
\else
\hbox to \myboxwidth{%
\unhbox\onelinebox
}%
\fi
\else
\hbox to \myboxwidth{%
$\Pi_{10,21}=\hbox{GN}_{54}$ spans $L_{7.7}$%
\hfil}%
\hbox to \myboxwidth{%
$3\infty\slashinfty\infty3\infty3\slashinfty3\infty\rtimes D_{2}$%
\hfil}%
\box\matricesbox
\fi
}%
\hfill\discretionary{}{}{}%
\setbox\matricesbox=\hbox{%
{$\left[\!\llap{\phantom{%
\begingroup \smaller\smaller\smaller\begin{tabular}{@{}c@{}}%
\phantom{0}\\\phantom{0}\\\phantom{0}
\end{tabular}\endgroup%
}}\right.$}%
\begingroup \smaller\smaller\smaller\begin{tabular}{@{}c@{}}%
-3\\\phantom{0}\\\phantom{0}
\end{tabular}\endgroup%
\kern3pt%
\begingroup \smaller\smaller\smaller\begin{tabular}{@{}c@{}}%
\phantom{0}\\1\\\phantom{0}
\end{tabular}\endgroup%
\kern3pt%
\begingroup \smaller\smaller\smaller\begin{tabular}{@{}c@{}}%
\phantom{0}\\\phantom{0}\\3
\end{tabular}\endgroup%
{$\left.\llap{\phantom{%
\begingroup \smaller\smaller\smaller\begin{tabular}{@{}c@{}}%
\phantom{0}\\\phantom{0}\\\phantom{0}
\end{tabular}\endgroup%
}}\!\right]$}%
{$\left[\!\llap{\phantom{%
\begingroup \smaller\smaller\smaller\begin{tabular}{@{}c@{}}%
0\\0\\0
\end{tabular}\endgroup%
}}\right.$}%
\begingroup \smaller\smaller\smaller\begin{tabular}{@{}c@{}}%
4\\-7\\1
\end{tabular}\endgroup%
\kern3pt%
\begingroup \smaller\smaller\smaller\begin{tabular}{@{}c@{}}%
1\\-1\\1
\end{tabular}\endgroup%
\kern3pt%
\begingroup \smaller\smaller\smaller\begin{tabular}{@{}c@{}}%
4\\2\\4
\end{tabular}\endgroup%
\kern3pt%
\begingroup \smaller\smaller\smaller\begin{tabular}{@{}c@{}}%
4\\5\\3
\end{tabular}\endgroup%
\kern3pt%
\begingroup \smaller\smaller\smaller\begin{tabular}{@{}c@{}}%
4\\7\\1
\end{tabular}\endgroup%
{$\left.\llap{\phantom{%
\begingroup \smaller\smaller\smaller\begin{tabular}{@{}c@{}}%
0\\0\\0
\end{tabular}\endgroup%
}}\!\right]$}%
}%
\ifdim\wd\matricesbox>\halfwidth\myboxwidth=\hsize\else\myboxwidth=\halfwidth\fi
\vbox{%
\ifdim\myboxwidth=\hsize
\setbox\onelinebox=\hbox{%
\vbox{\hbox{%
$\Pi_{10,22}=\hbox{GN}_{55}$ spans $L_{7.7}$%
}\hbox{%
$3\infty\infty\slashthree\infty\infty3\infty\slashthree\infty\rtimes D_{2}$%
}%
}%
\hfill\copy\matricesbox
}%
\ifdim\wd\onelinebox>\myboxwidth
\hbox to \myboxwidth{%
$\Pi_{10,22}=\hbox{GN}_{55}$ spans $L_{7.7}$%
\hfil
$3\infty\infty\slashthree\infty\infty3\infty\slashthree\infty\rtimes D_{2}$%
}%
\box\matricesbox
\else
\hbox to \myboxwidth{%
\unhbox\onelinebox
}%
\fi
\else
\hbox to \myboxwidth{%
$\Pi_{10,22}=\hbox{GN}_{55}$ spans $L_{7.7}$%
\hfil}%
\hbox to \myboxwidth{%
$3\infty\infty\slashthree\infty\infty3\infty\slashthree\infty\rtimes D_{2}$%
\hfil}%
\box\matricesbox
\fi
}%
\hfill\discretionary{}{}{}%
\setbox\matricesbox=\hbox{%
{$\left[\!\llap{\phantom{%
\begingroup \smaller\smaller\smaller\begin{tabular}{@{}c@{}}%
\phantom{0}\\\phantom{0}\\\phantom{0}
\end{tabular}\endgroup%
}}\right.$}%
\begingroup \smaller\smaller\smaller\begin{tabular}{@{}c@{}}%
-4\\\phantom{0}\\\phantom{0}
\end{tabular}\endgroup%
\kern3pt%
\begingroup \smaller\smaller\smaller\begin{tabular}{@{}c@{}}%
\phantom{0}\\1\\\phantom{0}
\end{tabular}\endgroup%
\kern3pt%
\begingroup \smaller\smaller\smaller\begin{tabular}{@{}c@{}}%
\phantom{0}\\\phantom{0}\\1
\end{tabular}\endgroup%
{$\left.\llap{\phantom{%
\begingroup \smaller\smaller\smaller\begin{tabular}{@{}c@{}}%
\phantom{0}\\\phantom{0}\\\phantom{0}
\end{tabular}\endgroup%
}}\!\right]$}%
{$\left[\!\llap{\phantom{%
\begingroup \smaller\smaller\smaller\begin{tabular}{@{}c@{}}%
0\\0\\0
\end{tabular}\endgroup%
}}\right.$}%
\begingroup \smaller\smaller\smaller\begin{tabular}{@{}c@{}}%
1\\2\\1
\end{tabular}\endgroup%
\kern3pt%
\begingroup \smaller\smaller\smaller\begin{tabular}{@{}c@{}}%
1\\1\\2
\end{tabular}\endgroup%
\kern3pt%
\begingroup \smaller\smaller\smaller\begin{tabular}{@{}c@{}}%
4\\-2\\8
\end{tabular}\endgroup%
\kern3pt%
\begingroup \smaller\smaller\smaller\begin{tabular}{@{}c@{}}%
2\\-3\\3
\end{tabular}\endgroup%
\kern3pt%
\begingroup \smaller\smaller\smaller\begin{tabular}{@{}c@{}}%
4\\-8\\2
\end{tabular}\endgroup%
{$\left.\llap{\phantom{%
\begingroup \smaller\smaller\smaller\begin{tabular}{@{}c@{}}%
0\\0\\0
\end{tabular}\endgroup%
}}\!\right]$}%
}%
\ifdim\wd\matricesbox>\halfwidth\myboxwidth=\hsize\else\myboxwidth=\halfwidth\fi
\vbox{%
\ifdim\myboxwidth=\hsize
\setbox\onelinebox=\hbox{%
\vbox{\hbox{%
$\Pi_{10,23}$ spans $L_{145.1}$%
}\hbox{%
$4\infty2\slashinfty2\infty44\slashinfty4\rtimes D_{2}$%
}%
}%
\hfill\copy\matricesbox
}%
\ifdim\wd\onelinebox>\myboxwidth
\hbox to \myboxwidth{%
$\Pi_{10,23}$ spans $L_{145.1}$%
\hfil
$4\infty2\slashinfty2\infty44\slashinfty4\rtimes D_{2}$%
}%
\box\matricesbox
\else
\hbox to \myboxwidth{%
\unhbox\onelinebox
}%
\fi
\else
\hbox to \myboxwidth{%
$\Pi_{10,23}$ spans $L_{145.1}$%
\hfil}%
\hbox to \myboxwidth{%
$4\infty2\slashinfty2\infty44\slashinfty4\rtimes D_{2}$%
\hfil}%
\box\matricesbox
\fi
}%
\hfill\discretionary{}{}{}%
\setbox\matricesbox=\hbox{%
{$\left[\!\llap{\phantom{%
\begingroup \smaller\smaller\smaller\begin{tabular}{@{}c@{}}%
\phantom{0}\\\phantom{0}\\\phantom{0}
\end{tabular}\endgroup%
}}\right.$}%
\begingroup \smaller\smaller\smaller\begin{tabular}{@{}c@{}}%
-1\\\phantom{0}\\\phantom{0}
\end{tabular}\endgroup%
\kern3pt%
\begingroup \smaller\smaller\smaller\begin{tabular}{@{}c@{}}%
\phantom{0}\\3/2\\\phantom{0}
\end{tabular}\endgroup%
\kern3pt%
\begingroup \smaller\smaller\smaller\begin{tabular}{@{}c@{}}%
\phantom{0}\\\phantom{0}\\9/2
\end{tabular}\endgroup%
{$\left.\llap{\phantom{%
\begingroup \smaller\smaller\smaller\begin{tabular}{@{}c@{}}%
\phantom{0}\\\phantom{0}\\\phantom{0}
\end{tabular}\endgroup%
}}\!\right]$}%
{$\left[\!\llap{\phantom{%
\begingroup \smaller\smaller\smaller\begin{tabular}{@{}c@{}}%
0\\0\\0
\end{tabular}\endgroup%
}}\right.$}%
\begingroup \smaller\smaller\smaller\begin{tabular}{@{}c@{}}%
6\\-5\\-1
\end{tabular}\endgroup%
\kern3pt%
\begingroup \smaller\smaller\smaller\begin{tabular}{@{}c@{}}%
2\\-1\\-1
\end{tabular}\endgroup%
\kern3pt%
\begingroup \smaller\smaller\smaller\begin{tabular}{@{}c@{}}%
6\\1\\-3
\end{tabular}\endgroup%
\kern3pt%
\begingroup \smaller\smaller\smaller\begin{tabular}{@{}c@{}}%
18\\9\\-7
\end{tabular}\endgroup%
\kern3pt%
\begingroup \smaller\smaller\smaller\begin{tabular}{@{}c@{}}%
6\\5\\-1
\end{tabular}\endgroup%
{$\left.\llap{\phantom{%
\begingroup \smaller\smaller\smaller\begin{tabular}{@{}c@{}}%
0\\0\\0
\end{tabular}\endgroup%
}}\!\right]$}%
}%
\ifdim\wd\matricesbox>\halfwidth\myboxwidth=\hsize\else\myboxwidth=\halfwidth\fi
\vbox{%
\ifdim\myboxwidth=\hsize
\setbox\onelinebox=\hbox{%
\vbox{\hbox{%
$\Pi_{10,24}$ spans $L_{168.1}$%
}\hbox{%
$\slashthree2226\slashthree6222\rtimes D_{2}$%
}%
}%
\hfill\copy\matricesbox
}%
\ifdim\wd\onelinebox>\myboxwidth
\hbox to \myboxwidth{%
$\Pi_{10,24}$ spans $L_{168.1}$%
\hfil
$\slashthree2226\slashthree6222\rtimes D_{2}$%
}%
\box\matricesbox
\else
\hbox to \myboxwidth{%
\unhbox\onelinebox
}%
\fi
\else
\hbox to \myboxwidth{%
$\Pi_{10,24}$ spans $L_{168.1}$%
\hfil}%
\hbox to \myboxwidth{%
$\slashthree2226\slashthree6222\rtimes D_{2}$%
\hfil}%
\box\matricesbox
\fi
}%
\hfill\discretionary{}{}{}%
\setbox\matricesbox=\hbox{%
{$\left[\!\llap{\phantom{%
\begingroup \smaller\smaller\smaller\begin{tabular}{@{}c@{}}%
\phantom{0}\\\phantom{0}\\\phantom{0}
\end{tabular}\endgroup%
}}\right.$}%
\begingroup \smaller\smaller\smaller\begin{tabular}{@{}c@{}}%
-1\\\phantom{0}\\\phantom{0}
\end{tabular}\endgroup%
\kern3pt%
\begingroup \smaller\smaller\smaller\begin{tabular}{@{}c@{}}%
\phantom{0}\\3/2\\\phantom{0}
\end{tabular}\endgroup%
\kern3pt%
\begingroup \smaller\smaller\smaller\begin{tabular}{@{}c@{}}%
\phantom{0}\\\phantom{0}\\9/2
\end{tabular}\endgroup%
{$\left.\llap{\phantom{%
\begingroup \smaller\smaller\smaller\begin{tabular}{@{}c@{}}%
\phantom{0}\\\phantom{0}\\\phantom{0}
\end{tabular}\endgroup%
}}\!\right]$}%
{$\left[\!\llap{\phantom{%
\begingroup \smaller\smaller\smaller\begin{tabular}{@{}c@{}}%
0\\0\\0
\end{tabular}\endgroup%
}}\right.$}%
\begingroup \smaller\smaller\smaller\begin{tabular}{@{}c@{}}%
2\\2\\0
\end{tabular}\endgroup%
\kern3pt%
\begingroup \smaller\smaller\smaller\begin{tabular}{@{}c@{}}%
6\\4\\2
\end{tabular}\endgroup%
\kern3pt%
\begingroup \smaller\smaller\smaller\begin{tabular}{@{}c@{}}%
6\\1\\3
\end{tabular}\endgroup%
\kern3pt%
\begingroup \smaller\smaller\smaller\begin{tabular}{@{}c@{}}%
18\\-6\\8
\end{tabular}\endgroup%
\kern3pt%
\begingroup \smaller\smaller\smaller\begin{tabular}{@{}c@{}}%
6\\-4\\2
\end{tabular}\endgroup%
\kern3pt%
\begingroup \smaller\smaller\smaller\begin{tabular}{@{}c@{}}%
2\\-2\\0
\end{tabular}\endgroup%
{$\left.\llap{\phantom{%
\begingroup \smaller\smaller\smaller\begin{tabular}{@{}c@{}}%
0\\0\\0
\end{tabular}\endgroup%
}}\!\right]$}%
}%
\ifdim\wd\matricesbox>\halfwidth\myboxwidth=\hsize\else\myboxwidth=\halfwidth\fi
\vbox{%
\ifdim\myboxwidth=\hsize
\setbox\onelinebox=\hbox{%
\vbox{\hbox{%
$\Pi_{10,25}$ spans $L_{155.1}$%
}\hbox{%
$32|23622|226\rtimes D_{2}$%
}%
}%
\hfill\copy\matricesbox
}%
\ifdim\wd\onelinebox>\myboxwidth
\hbox to \myboxwidth{%
$\Pi_{10,25}$ spans $L_{155.1}$%
\hfil
$32|23622|226\rtimes D_{2}$%
}%
\box\matricesbox
\else
\hbox to \myboxwidth{%
\unhbox\onelinebox
}%
\fi
\else
\hbox to \myboxwidth{%
$\Pi_{10,25}$ spans $L_{155.1}$%
\hfil}%
\hbox to \myboxwidth{%
$32|23622|226\rtimes D_{2}$%
\hfil}%
\box\matricesbox
\fi
}%
\hfill\discretionary{}{}{}%
\setbox\matricesbox=\hbox{%
{$\left[\!\llap{\phantom{%
\begingroup \smaller\smaller\smaller\begin{tabular}{@{}c@{}}%
\phantom{0}\\\phantom{0}\\\phantom{0}
\end{tabular}\endgroup%
}}\right.$}%
\begingroup \smaller\smaller\smaller\begin{tabular}{@{}c@{}}%
-1/8\\\phantom{0}\\\phantom{0}
\end{tabular}\endgroup%
\kern3pt%
\begingroup \smaller\smaller\smaller\begin{tabular}{@{}c@{}}%
\phantom{0}\\77/2\\-7/2
\end{tabular}\endgroup%
\kern3pt%
\begingroup \smaller\smaller\smaller\begin{tabular}{@{}c@{}}%
\phantom{0}\\-7/2\\77/2
\end{tabular}\endgroup%
{$\left.\llap{\phantom{%
\begingroup \smaller\smaller\smaller\begin{tabular}{@{}c@{}}%
\phantom{0}\\\phantom{0}\\\phantom{0}
\end{tabular}\endgroup%
}}\!\right]$}%
{$\left[\!\llap{\phantom{%
\begingroup \smaller\smaller\smaller\begin{tabular}{@{}c@{}}%
0\\0\\0
\end{tabular}\endgroup%
}}\right.$}%
\begingroup \smaller\smaller\smaller\begin{tabular}{@{}c@{}}%
20\\-1\\-1
\end{tabular}\endgroup%
\kern3pt%
\begingroup \smaller\smaller\smaller\begin{tabular}{@{}c@{}}%
84\\-1\\-5
\end{tabular}\endgroup%
\kern3pt%
\begingroup \smaller\smaller\smaller\begin{tabular}{@{}c@{}}%
70\\1\\-4
\end{tabular}\endgroup%
\kern3pt%
\begingroup \smaller\smaller\smaller\begin{tabular}{@{}c@{}}%
48\\2\\-2
\end{tabular}\endgroup%
\kern3pt%
\begingroup \smaller\smaller\smaller\begin{tabular}{@{}c@{}}%
14\\1\\0
\end{tabular}\endgroup%
{$\left.\llap{\phantom{%
\begingroup \smaller\smaller\smaller\begin{tabular}{@{}c@{}}%
0\\0\\0
\end{tabular}\endgroup%
}}\!\right]$}%
}%
\ifdim\wd\matricesbox>\halfwidth\myboxwidth=\hsize\else\myboxwidth=\halfwidth\fi
\vbox{%
\ifdim\myboxwidth=\hsize
\setbox\onelinebox=\hbox{%
\vbox{\hbox{%
$\Pi_{10,26}$ spans $L_{69.11}$%
}\hbox{%
$2222222222\rtimes C_{2}$%
}%
}%
\hfill\copy\matricesbox
}%
\ifdim\wd\onelinebox>\myboxwidth
\hbox to \myboxwidth{%
$\Pi_{10,26}$ spans $L_{69.11}$%
\hfil
$2222222222\rtimes C_{2}$%
}%
\box\matricesbox
\else
\hbox to \myboxwidth{%
\unhbox\onelinebox
}%
\fi
\else
\hbox to \myboxwidth{%
$\Pi_{10,26}$ spans $L_{69.11}$%
\hfil}%
\hbox to \myboxwidth{%
$2222222222\rtimes C_{2}$%
\hfil}%
\box\matricesbox
\fi
}%
\hfill\discretionary{}{}{}%
\setbox\matricesbox=\hbox{%
{$\left[\!\llap{\phantom{%
\begingroup \smaller\smaller\smaller\begin{tabular}{@{}c@{}}%
\phantom{0}\\\phantom{0}\\\phantom{0}
\end{tabular}\endgroup%
}}\right.$}%
\begingroup \smaller\smaller\smaller\begin{tabular}{@{}c@{}}%
-1/2\\\phantom{0}\\\phantom{0}
\end{tabular}\endgroup%
\kern3pt%
\begingroup \smaller\smaller\smaller\begin{tabular}{@{}c@{}}%
\phantom{0}\\15/2\\\phantom{0}
\end{tabular}\endgroup%
\kern3pt%
\begingroup \smaller\smaller\smaller\begin{tabular}{@{}c@{}}%
\phantom{0}\\\phantom{0}\\10
\end{tabular}\endgroup%
{$\left.\llap{\phantom{%
\begingroup \smaller\smaller\smaller\begin{tabular}{@{}c@{}}%
\phantom{0}\\\phantom{0}\\\phantom{0}
\end{tabular}\endgroup%
}}\!\right]$}%
{$\left[\!\llap{\phantom{%
\begingroup \smaller\smaller\smaller\begin{tabular}{@{}c@{}}%
0\\0\\0
\end{tabular}\endgroup%
}}\right.$}%
\begingroup \smaller\smaller\smaller\begin{tabular}{@{}c@{}}%
3\\-1\\0
\end{tabular}\endgroup%
\kern3pt%
\begingroup \smaller\smaller\smaller\begin{tabular}{@{}c@{}}%
5\\-1\\-1
\end{tabular}\endgroup%
\kern3pt%
\begingroup \smaller\smaller\smaller\begin{tabular}{@{}c@{}}%
8\\0\\-2
\end{tabular}\endgroup%
\kern3pt%
\begingroup \smaller\smaller\smaller\begin{tabular}{@{}c@{}}%
30\\4\\-6
\end{tabular}\endgroup%
\kern3pt%
\begingroup \smaller\smaller\smaller\begin{tabular}{@{}c@{}}%
40\\8\\-6
\end{tabular}\endgroup%
{$\left.\llap{\phantom{%
\begingroup \smaller\smaller\smaller\begin{tabular}{@{}c@{}}%
0\\0\\0
\end{tabular}\endgroup%
}}\!\right]$}%
}%
\ifdim\wd\matricesbox>\halfwidth\myboxwidth=\hsize\else\myboxwidth=\halfwidth\fi
\vbox{%
\ifdim\myboxwidth=\hsize
\setbox\onelinebox=\hbox{%
\vbox{\hbox{%
$\Pi_{10,27}$ spans $L_{17.11}$%
}\hbox{%
$2222222222\rtimes C_{2}$%
}%
}%
\hfill\copy\matricesbox
}%
\ifdim\wd\onelinebox>\myboxwidth
\hbox to \myboxwidth{%
$\Pi_{10,27}$ spans $L_{17.11}$%
\hfil
$2222222222\rtimes C_{2}$%
}%
\box\matricesbox
\else
\hbox to \myboxwidth{%
\unhbox\onelinebox
}%
\fi
\else
\hbox to \myboxwidth{%
$\Pi_{10,27}$ spans $L_{17.11}$%
\hfil}%
\hbox to \myboxwidth{%
$2222222222\rtimes C_{2}$%
\hfil}%
\box\matricesbox
\fi
}%
\hfill\discretionary{}{}{}%
\setbox\matricesbox=\hbox{%
{$\left[\!\llap{\phantom{%
\begingroup \smaller\smaller\smaller\begin{tabular}{@{}c@{}}%
\phantom{0}\\\phantom{0}\\\phantom{0}
\end{tabular}\endgroup%
}}\right.$}%
\begingroup \smaller\smaller\smaller\begin{tabular}{@{}c@{}}%
-1\\\phantom{0}\\\phantom{0}
\end{tabular}\endgroup%
\kern3pt%
\begingroup \smaller\smaller\smaller\begin{tabular}{@{}c@{}}%
\phantom{0}\\6\\-3
\end{tabular}\endgroup%
\kern3pt%
\begingroup \smaller\smaller\smaller\begin{tabular}{@{}c@{}}%
\phantom{0}\\-3\\6
\end{tabular}\endgroup%
{$\left.\llap{\phantom{%
\begingroup \smaller\smaller\smaller\begin{tabular}{@{}c@{}}%
\phantom{0}\\\phantom{0}\\\phantom{0}
\end{tabular}\endgroup%
}}\!\right]$}%
{$\left[\!\llap{\phantom{%
\begingroup \smaller\smaller\smaller\begin{tabular}{@{}c@{}}%
0\\0\\0
\end{tabular}\endgroup%
}}\right.$}%
\begingroup \smaller\smaller\smaller\begin{tabular}{@{}c@{}}%
6\\-3\\-2
\end{tabular}\endgroup%
\kern3pt%
\begingroup \smaller\smaller\smaller\begin{tabular}{@{}c@{}}%
6\\-2\\-3
\end{tabular}\endgroup%
\kern3pt%
\begingroup \smaller\smaller\smaller\begin{tabular}{@{}c@{}}%
2\\0\\-1
\end{tabular}\endgroup%
\kern3pt%
\begingroup \smaller\smaller\smaller\begin{tabular}{@{}c@{}}%
6\\2\\-1
\end{tabular}\endgroup%
\kern3pt%
\begingroup \smaller\smaller\smaller\begin{tabular}{@{}c@{}}%
18\\8\\1
\end{tabular}\endgroup%
{$\left.\llap{\phantom{%
\begingroup \smaller\smaller\smaller\begin{tabular}{@{}c@{}}%
0\\0\\0
\end{tabular}\endgroup%
}}\!\right]$}%
}%
\ifdim\wd\matricesbox>\halfwidth\myboxwidth=\hsize\else\myboxwidth=\halfwidth\fi
\vbox{%
\ifdim\myboxwidth=\hsize
\setbox\onelinebox=\hbox{%
\vbox{\hbox{%
$\Pi_{10,28}$ spans $L_{168.1}$%
}\hbox{%
$3222632226\rtimes C_{2}$%
}%
}%
\hfill\copy\matricesbox
}%
\ifdim\wd\onelinebox>\myboxwidth
\hbox to \myboxwidth{%
$\Pi_{10,28}$ spans $L_{168.1}$%
\hfil
$3222632226\rtimes C_{2}$%
}%
\box\matricesbox
\else
\hbox to \myboxwidth{%
\unhbox\onelinebox
}%
\fi
\else
\hbox to \myboxwidth{%
$\Pi_{10,28}$ spans $L_{168.1}$%
\hfil}%
\hbox to \myboxwidth{%
$3222632226\rtimes C_{2}$%
\hfil}%
\box\matricesbox
\fi
}%
\hfill\discretionary{}{}{}%
\setbox\matricesbox=\hbox{%
{$\left[\!\llap{\phantom{%
\begingroup \smaller\smaller\smaller\begin{tabular}{@{}c@{}}%
\phantom{0}\\\phantom{0}\\\phantom{0}
\end{tabular}\endgroup%
}}\right.$}%
\begingroup \smaller\smaller\smaller\begin{tabular}{@{}c@{}}%
-1\\\phantom{0}\\\phantom{0}
\end{tabular}\endgroup%
\kern3pt%
\begingroup \smaller\smaller\smaller\begin{tabular}{@{}c@{}}%
\phantom{0}\\6\\\phantom{0}
\end{tabular}\endgroup%
\kern3pt%
\begingroup \smaller\smaller\smaller\begin{tabular}{@{}c@{}}%
\phantom{0}\\\phantom{0}\\6
\end{tabular}\endgroup%
{$\left.\llap{\phantom{%
\begingroup \smaller\smaller\smaller\begin{tabular}{@{}c@{}}%
\phantom{0}\\\phantom{0}\\\phantom{0}
\end{tabular}\endgroup%
}}\!\right]$}%
{$\left[\!\llap{\phantom{%
\begingroup \smaller\smaller\smaller\begin{tabular}{@{}c@{}}%
0\\0\\0
\end{tabular}\endgroup%
}}\right.$}%
\begingroup \smaller\smaller\smaller\begin{tabular}{@{}c@{}}%
2\\0\\1
\end{tabular}\endgroup%
\kern3pt%
\begingroup \smaller\smaller\smaller\begin{tabular}{@{}c@{}}%
3\\1\\1
\end{tabular}\endgroup%
\kern3pt%
\begingroup \smaller\smaller\smaller\begin{tabular}{@{}c@{}}%
2\\1\\0
\end{tabular}\endgroup%
\kern3pt%
\begingroup \smaller\smaller\smaller\begin{tabular}{@{}c@{}}%
3\\1\\-1
\end{tabular}\endgroup%
\kern3pt%
\begingroup \smaller\smaller\smaller\begin{tabular}{@{}c@{}}%
2\\0\\-1
\end{tabular}\endgroup%
\kern3pt%
\begingroup \smaller\smaller\smaller\begin{tabular}{@{}c@{}}%
3\\-1\\-1
\end{tabular}\endgroup%
\kern3pt%
\begingroup \smaller\smaller\smaller\begin{tabular}{@{}c@{}}%
12\\-5\\-1
\end{tabular}\endgroup%
\kern3pt%
\begingroup \smaller\smaller\smaller\begin{tabular}{@{}c@{}}%
12\\-5\\1
\end{tabular}\endgroup%
\kern3pt%
\begingroup \smaller\smaller\smaller\begin{tabular}{@{}c@{}}%
24\\-8\\6
\end{tabular}\endgroup%
\kern3pt%
\begingroup \smaller\smaller\smaller\begin{tabular}{@{}c@{}}%
24\\-6\\8
\end{tabular}\endgroup%
{$\left.\llap{\phantom{%
\begingroup \smaller\smaller\smaller\begin{tabular}{@{}c@{}}%
0\\0\\0
\end{tabular}\endgroup%
}}\!\right]$}%
}%
\ifdim\wd\matricesbox>\halfwidth\myboxwidth=\hsize\else\myboxwidth=\halfwidth\fi
\vbox{%
\ifdim\myboxwidth=\hsize
\setbox\onelinebox=\hbox{%
\vbox{\hbox{%
$\Pi_{10,29}$ spans $L_{123.8}$%
}\hbox{%
$2222222422$%
}%
}%
\hfill\copy\matricesbox
}%
\ifdim\wd\onelinebox>\myboxwidth
\hbox to \myboxwidth{%
$\Pi_{10,29}$ spans $L_{123.8}$%
\hfil
$2222222422$%
}%
\box\matricesbox
\else
\hbox to \myboxwidth{%
\unhbox\onelinebox
}%
\fi
\else
\hbox to \myboxwidth{%
$\Pi_{10,29}$ spans $L_{123.8}$%
\hfil}%
\hbox to \myboxwidth{%
$2222222422$%
\hfil}%
\box\matricesbox
\fi
}%
\hfill\discretionary{}{}{}%
\setbox\matricesbox=\hbox{%
{$\left[\!\llap{\phantom{%
\begingroup \smaller\smaller\smaller\begin{tabular}{@{}c@{}}%
\phantom{0}\\\phantom{0}\\\phantom{0}
\end{tabular}\endgroup%
}}\right.$}%
\begingroup \smaller\smaller\smaller\begin{tabular}{@{}c@{}}%
-1\\\phantom{0}\\\phantom{0}
\end{tabular}\endgroup%
\kern3pt%
\begingroup \smaller\smaller\smaller\begin{tabular}{@{}c@{}}%
\phantom{0}\\6\\\phantom{0}
\end{tabular}\endgroup%
\kern3pt%
\begingroup \smaller\smaller\smaller\begin{tabular}{@{}c@{}}%
\phantom{0}\\\phantom{0}\\6
\end{tabular}\endgroup%
{$\left.\llap{\phantom{%
\begingroup \smaller\smaller\smaller\begin{tabular}{@{}c@{}}%
\phantom{0}\\\phantom{0}\\\phantom{0}
\end{tabular}\endgroup%
}}\!\right]$}%
{$\left[\!\llap{\phantom{%
\begingroup \smaller\smaller\smaller\begin{tabular}{@{}c@{}}%
0\\0\\0
\end{tabular}\endgroup%
}}\right.$}%
\begingroup \smaller\smaller\smaller\begin{tabular}{@{}c@{}}%
2\\0\\-1
\end{tabular}\endgroup%
\kern3pt%
\begingroup \smaller\smaller\smaller\begin{tabular}{@{}c@{}}%
24\\-6\\-8
\end{tabular}\endgroup%
\kern3pt%
\begingroup \smaller\smaller\smaller\begin{tabular}{@{}c@{}}%
24\\-8\\-6
\end{tabular}\endgroup%
\kern3pt%
\begingroup \smaller\smaller\smaller\begin{tabular}{@{}c@{}}%
2\\-1\\0
\end{tabular}\endgroup%
\kern3pt%
\begingroup \smaller\smaller\smaller\begin{tabular}{@{}c@{}}%
3\\-1\\1
\end{tabular}\endgroup%
\kern3pt%
\begingroup \smaller\smaller\smaller\begin{tabular}{@{}c@{}}%
2\\0\\1
\end{tabular}\endgroup%
\kern3pt%
\begingroup \smaller\smaller\smaller\begin{tabular}{@{}c@{}}%
3\\1\\1
\end{tabular}\endgroup%
\kern3pt%
\begingroup \smaller\smaller\smaller\begin{tabular}{@{}c@{}}%
12\\5\\1
\end{tabular}\endgroup%
\kern3pt%
\begingroup \smaller\smaller\smaller\begin{tabular}{@{}c@{}}%
12\\5\\-1
\end{tabular}\endgroup%
\kern3pt%
\begingroup \smaller\smaller\smaller\begin{tabular}{@{}c@{}}%
3\\1\\-1
\end{tabular}\endgroup%
{$\left.\llap{\phantom{%
\begingroup \smaller\smaller\smaller\begin{tabular}{@{}c@{}}%
0\\0\\0
\end{tabular}\endgroup%
}}\!\right]$}%
}%
\ifdim\wd\matricesbox>\halfwidth\myboxwidth=\hsize\else\myboxwidth=\halfwidth\fi
\vbox{%
\ifdim\myboxwidth=\hsize
\setbox\onelinebox=\hbox{%
\vbox{\hbox{%
$\Pi_{10,30}$ spans $L_{123.8}$%
}\hbox{%
$2222222222$%
}%
}%
\hfill\copy\matricesbox
}%
\ifdim\wd\onelinebox>\myboxwidth
\hbox to \myboxwidth{%
$\Pi_{10,30}$ spans $L_{123.8}$%
\hfil
$2222222222$%
}%
\box\matricesbox
\else
\hbox to \myboxwidth{%
\unhbox\onelinebox
}%
\fi
\else
\hbox to \myboxwidth{%
$\Pi_{10,30}$ spans $L_{123.8}$%
\hfil}%
\hbox to \myboxwidth{%
$2222222222$%
\hfil}%
\box\matricesbox
\fi
}%
\hfill\discretionary{}{}{}%
\setbox\matricesbox=\hbox{%
{$\left[\!\llap{\phantom{%
\begingroup \smaller\smaller\smaller\begin{tabular}{@{}c@{}}%
\phantom{0}\\\phantom{0}\\\phantom{0}
\end{tabular}\endgroup%
}}\right.$}%
\begingroup \smaller\smaller\smaller\begin{tabular}{@{}c@{}}%
-1\\\phantom{0}\\\phantom{0}
\end{tabular}\endgroup%
\kern3pt%
\begingroup \smaller\smaller\smaller\begin{tabular}{@{}c@{}}%
\phantom{0}\\6\\-3
\end{tabular}\endgroup%
\kern3pt%
\begingroup \smaller\smaller\smaller\begin{tabular}{@{}c@{}}%
\phantom{0}\\-3\\6
\end{tabular}\endgroup%
{$\left.\llap{\phantom{%
\begingroup \smaller\smaller\smaller\begin{tabular}{@{}c@{}}%
\phantom{0}\\\phantom{0}\\\phantom{0}
\end{tabular}\endgroup%
}}\!\right]$}%
{$\left[\!\llap{\phantom{%
\begingroup \smaller\smaller\smaller\begin{tabular}{@{}c@{}}%
0\\0\\0
\end{tabular}\endgroup%
}}\right.$}%
\begingroup \smaller\smaller\smaller\begin{tabular}{@{}c@{}}%
6\\-1\\2
\end{tabular}\endgroup%
\kern3pt%
\begingroup \smaller\smaller\smaller\begin{tabular}{@{}c@{}}%
6\\1\\3
\end{tabular}\endgroup%
\kern3pt%
\begingroup \smaller\smaller\smaller\begin{tabular}{@{}c@{}}%
2\\1\\1
\end{tabular}\endgroup%
\kern3pt%
\begingroup \smaller\smaller\smaller\begin{tabular}{@{}c@{}}%
2\\1\\0
\end{tabular}\endgroup%
\kern3pt%
\begingroup \smaller\smaller\smaller\begin{tabular}{@{}c@{}}%
6\\1\\-2
\end{tabular}\endgroup%
\kern3pt%
\begingroup \smaller\smaller\smaller\begin{tabular}{@{}c@{}}%
18\\-1\\-8
\end{tabular}\endgroup%
\kern3pt%
\begingroup \smaller\smaller\smaller\begin{tabular}{@{}c@{}}%
6\\-2\\-3
\end{tabular}\endgroup%
\kern3pt%
\begingroup \smaller\smaller\smaller\begin{tabular}{@{}c@{}}%
18\\-8\\-7
\end{tabular}\endgroup%
\kern3pt%
\begingroup \smaller\smaller\smaller\begin{tabular}{@{}c@{}}%
6\\-3\\-1
\end{tabular}\endgroup%
\kern3pt%
\begingroup \smaller\smaller\smaller\begin{tabular}{@{}c@{}}%
18\\-7\\1
\end{tabular}\endgroup%
{$\left.\llap{\phantom{%
\begingroup \smaller\smaller\smaller\begin{tabular}{@{}c@{}}%
0\\0\\0
\end{tabular}\endgroup%
}}\!\right]$}%
}%
\ifdim\wd\matricesbox>\halfwidth\myboxwidth=\hsize\else\myboxwidth=\halfwidth\fi
\vbox{%
\ifdim\myboxwidth=\hsize
\setbox\onelinebox=\hbox{%
\vbox{\hbox{%
$\Pi_{10,31}$ spans $L_{155.1}$%
}\hbox{%
$3232262626$%
}%
}%
\hfill\copy\matricesbox
}%
\ifdim\wd\onelinebox>\myboxwidth
\hbox to \myboxwidth{%
$\Pi_{10,31}$ spans $L_{155.1}$%
\hfil
$3232262626$%
}%
\box\matricesbox
\else
\hbox to \myboxwidth{%
\unhbox\onelinebox
}%
\fi
\else
\hbox to \myboxwidth{%
$\Pi_{10,31}$ spans $L_{155.1}$%
\hfil}%
\hbox to \myboxwidth{%
$3232262626$%
\hfil}%
\box\matricesbox
\fi
}%
\hfill\discretionary{}{}{}%
\setbox\matricesbox=\hbox{%
{$\left[\!\llap{\phantom{%
\begingroup \smaller\smaller\smaller\begin{tabular}{@{}c@{}}%
\phantom{0}\\\phantom{0}\\\phantom{0}
\end{tabular}\endgroup%
}}\right.$}%
\begingroup \smaller\smaller\smaller\begin{tabular}{@{}c@{}}%
-1\\\phantom{0}\\\phantom{0}
\end{tabular}\endgroup%
\kern3pt%
\begingroup \smaller\smaller\smaller\begin{tabular}{@{}c@{}}%
\phantom{0}\\6\\-3
\end{tabular}\endgroup%
\kern3pt%
\begingroup \smaller\smaller\smaller\begin{tabular}{@{}c@{}}%
\phantom{0}\\-3\\6
\end{tabular}\endgroup%
{$\left.\llap{\phantom{%
\begingroup \smaller\smaller\smaller\begin{tabular}{@{}c@{}}%
\phantom{0}\\\phantom{0}\\\phantom{0}
\end{tabular}\endgroup%
}}\!\right]$}%
{$\left[\!\llap{\phantom{%
\begingroup \smaller\smaller\smaller\begin{tabular}{@{}c@{}}%
0\\0\\0
\end{tabular}\endgroup%
}}\right.$}%
\begingroup \smaller\smaller\smaller\begin{tabular}{@{}c@{}}%
6\\-1\\2
\end{tabular}\endgroup%
\kern3pt%
\begingroup \smaller\smaller\smaller\begin{tabular}{@{}c@{}}%
6\\1\\3
\end{tabular}\endgroup%
\kern3pt%
\begingroup \smaller\smaller\smaller\begin{tabular}{@{}c@{}}%
2\\1\\1
\end{tabular}\endgroup%
\kern3pt%
\begingroup \smaller\smaller\smaller\begin{tabular}{@{}c@{}}%
6\\3\\1
\end{tabular}\endgroup%
\kern3pt%
\begingroup \smaller\smaller\smaller\begin{tabular}{@{}c@{}}%
6\\2\\-1
\end{tabular}\endgroup%
\kern3pt%
\begingroup \smaller\smaller\smaller\begin{tabular}{@{}c@{}}%
2\\0\\-1
\end{tabular}\endgroup%
\kern3pt%
\begingroup \smaller\smaller\smaller\begin{tabular}{@{}c@{}}%
6\\-2\\-3
\end{tabular}\endgroup%
\kern3pt%
\begingroup \smaller\smaller\smaller\begin{tabular}{@{}c@{}}%
18\\-8\\-7
\end{tabular}\endgroup%
\kern3pt%
\begingroup \smaller\smaller\smaller\begin{tabular}{@{}c@{}}%
6\\-3\\-1
\end{tabular}\endgroup%
\kern3pt%
\begingroup \smaller\smaller\smaller\begin{tabular}{@{}c@{}}%
18\\-7\\1
\end{tabular}\endgroup%
{$\left.\llap{\phantom{%
\begingroup \smaller\smaller\smaller\begin{tabular}{@{}c@{}}%
0\\0\\0
\end{tabular}\endgroup%
}}\!\right]$}%
}%
\ifdim\wd\matricesbox>\halfwidth\myboxwidth=\hsize\else\myboxwidth=\halfwidth\fi
\vbox{%
\ifdim\myboxwidth=\hsize
\setbox\onelinebox=\hbox{%
\vbox{\hbox{%
$\Pi_{10,32}$ spans $L_{155.1}$%
}\hbox{%
$3223222626$%
}%
}%
\hfill\copy\matricesbox
}%
\ifdim\wd\onelinebox>\myboxwidth
\hbox to \myboxwidth{%
$\Pi_{10,32}$ spans $L_{155.1}$%
\hfil
$3223222626$%
}%
\box\matricesbox
\else
\hbox to \myboxwidth{%
\unhbox\onelinebox
}%
\fi
\else
\hbox to \myboxwidth{%
$\Pi_{10,32}$ spans $L_{155.1}$%
\hfil}%
\hbox to \myboxwidth{%
$3223222626$%
\hfil}%
\box\matricesbox
\fi
}%
\hfill\discretionary{}{}{}%
\setbox\matricesbox=\hbox{%
{$\left[\!\llap{\phantom{%
\begingroup \smaller\smaller\smaller\begin{tabular}{@{}c@{}}%
\phantom{0}\\\phantom{0}\\\phantom{0}
\end{tabular}\endgroup%
}}\right.$}%
\begingroup \smaller\smaller\smaller\begin{tabular}{@{}c@{}}%
-1\\\phantom{0}\\\phantom{0}
\end{tabular}\endgroup%
\kern3pt%
\begingroup \smaller\smaller\smaller\begin{tabular}{@{}c@{}}%
\phantom{0}\\6\\-3
\end{tabular}\endgroup%
\kern3pt%
\begingroup \smaller\smaller\smaller\begin{tabular}{@{}c@{}}%
\phantom{0}\\-3\\6
\end{tabular}\endgroup%
{$\left.\llap{\phantom{%
\begingroup \smaller\smaller\smaller\begin{tabular}{@{}c@{}}%
\phantom{0}\\\phantom{0}\\\phantom{0}
\end{tabular}\endgroup%
}}\!\right]$}%
{$\left[\!\llap{\phantom{%
\begingroup \smaller\smaller\smaller\begin{tabular}{@{}c@{}}%
0\\0\\0
\end{tabular}\endgroup%
}}\right.$}%
\begingroup \smaller\smaller\smaller\begin{tabular}{@{}c@{}}%
6\\-1\\-3
\end{tabular}\endgroup%
\kern3pt%
\begingroup \smaller\smaller\smaller\begin{tabular}{@{}c@{}}%
18\\-7\\-8
\end{tabular}\endgroup%
\kern3pt%
\begingroup \smaller\smaller\smaller\begin{tabular}{@{}c@{}}%
6\\-3\\-2
\end{tabular}\endgroup%
\kern3pt%
\begingroup \smaller\smaller\smaller\begin{tabular}{@{}c@{}}%
2\\-1\\0
\end{tabular}\endgroup%
\kern3pt%
\begingroup \smaller\smaller\smaller\begin{tabular}{@{}c@{}}%
2\\0\\1
\end{tabular}\endgroup%
\kern3pt%
\begingroup \smaller\smaller\smaller\begin{tabular}{@{}c@{}}%
6\\2\\3
\end{tabular}\endgroup%
\kern3pt%
\begingroup \smaller\smaller\smaller\begin{tabular}{@{}c@{}}%
18\\8\\7
\end{tabular}\endgroup%
\kern3pt%
\begingroup \smaller\smaller\smaller\begin{tabular}{@{}c@{}}%
6\\3\\1
\end{tabular}\endgroup%
\kern3pt%
\begingroup \smaller\smaller\smaller\begin{tabular}{@{}c@{}}%
6\\2\\-1
\end{tabular}\endgroup%
\kern3pt%
\begingroup \smaller\smaller\smaller\begin{tabular}{@{}c@{}}%
18\\1\\-7
\end{tabular}\endgroup%
{$\left.\llap{\phantom{%
\begingroup \smaller\smaller\smaller\begin{tabular}{@{}c@{}}%
0\\0\\0
\end{tabular}\endgroup%
}}\!\right]$}%
}%
\ifdim\wd\matricesbox>\halfwidth\myboxwidth=\hsize\else\myboxwidth=\halfwidth\fi
\vbox{%
\ifdim\myboxwidth=\hsize
\setbox\onelinebox=\hbox{%
\vbox{\hbox{%
$\Pi_{10,33}$ spans $L_{155.1}$%
}\hbox{%
$6223226362$%
}%
}%
\hfill\copy\matricesbox
}%
\ifdim\wd\onelinebox>\myboxwidth
\hbox to \myboxwidth{%
$\Pi_{10,33}$ spans $L_{155.1}$%
\hfil
$6223226362$%
}%
\box\matricesbox
\else
\hbox to \myboxwidth{%
\unhbox\onelinebox
}%
\fi
\else
\hbox to \myboxwidth{%
$\Pi_{10,33}$ spans $L_{155.1}$%
\hfil}%
\hbox to \myboxwidth{%
$6223226362$%
\hfil}%
\box\matricesbox
\fi
}%
\hfill\discretionary{}{}{}%

\vskip2pt\hrule\vskip2pt

\leavevmode\setbox\matricesbox=\hbox{%
{$\left[\!\llap{\phantom{%
\begingroup \smaller\smaller\smaller\begin{tabular}{@{}c@{}}%
\phantom{0}\\\phantom{0}\\\phantom{0}
\end{tabular}\endgroup%
}}\right.$}%
\begingroup \smaller\smaller\smaller\begin{tabular}{@{}c@{}}%
-1/4\\\phantom{0}\\\phantom{0}
\end{tabular}\endgroup%
\kern3pt%
\begingroup \smaller\smaller\smaller\begin{tabular}{@{}c@{}}%
\phantom{0}\\15\\\phantom{0}
\end{tabular}\endgroup%
\kern3pt%
\begingroup \smaller\smaller\smaller\begin{tabular}{@{}c@{}}%
\phantom{0}\\\phantom{0}\\15
\end{tabular}\endgroup%
{$\left.\llap{\phantom{%
\begingroup \smaller\smaller\smaller\begin{tabular}{@{}c@{}}%
\phantom{0}\\\phantom{0}\\\phantom{0}
\end{tabular}\endgroup%
}}\!\right]$}%
{$\left[\!\llap{\phantom{%
\begingroup \smaller\smaller\smaller\begin{tabular}{@{}c@{}}%
0\\0\\0
\end{tabular}\endgroup%
}}\right.$}%
\begingroup \smaller\smaller\smaller\begin{tabular}{@{}c@{}}%
6\\-1\\0
\end{tabular}\endgroup%
\kern3pt%
\begingroup \smaller\smaller\smaller\begin{tabular}{@{}c@{}}%
20\\-2\\2
\end{tabular}\endgroup%
\kern3pt%
\begingroup \smaller\smaller\smaller\begin{tabular}{@{}c@{}}%
30\\-1\\4
\end{tabular}\endgroup%
\kern3pt%
\begingroup \smaller\smaller\smaller\begin{tabular}{@{}c@{}}%
30\\1\\4
\end{tabular}\endgroup%
\kern3pt%
\begingroup \smaller\smaller\smaller\begin{tabular}{@{}c@{}}%
20\\2\\2
\end{tabular}\endgroup%
\kern3pt%
\begingroup \smaller\smaller\smaller\begin{tabular}{@{}c@{}}%
30\\4\\1
\end{tabular}\endgroup%
{$\left.\llap{\phantom{%
\begingroup \smaller\smaller\smaller\begin{tabular}{@{}c@{}}%
0\\0\\0
\end{tabular}\endgroup%
}}\!\right]$}%
}%
\ifdim\wd\matricesbox>\halfwidth\myboxwidth=\hsize\else\myboxwidth=\halfwidth\fi
\vbox{%
\ifdim\myboxwidth=\hsize
\setbox\onelinebox=\hbox{%
\vbox{\hbox{%
$\Pi_{11,1}$ spans $L_{19.10}$%
}\hbox{%
$|22222\slashtwo22222\rtimes D_{2}$%
}%
}%
\hfill\copy\matricesbox
}%
\ifdim\wd\onelinebox>\myboxwidth
\hbox to \myboxwidth{%
$\Pi_{11,1}$ spans $L_{19.10}$%
\hfil
$|22222\slashtwo22222\rtimes D_{2}$%
}%
\box\matricesbox
\else
\hbox to \myboxwidth{%
\unhbox\onelinebox
}%
\fi
\else
\hbox to \myboxwidth{%
$\Pi_{11,1}$ spans $L_{19.10}$%
\hfil}%
\hbox to \myboxwidth{%
$|22222\slashtwo22222\rtimes D_{2}$%
\hfil}%
\box\matricesbox
\fi
}%
\hfill\discretionary{}{}{}%
\setbox\matricesbox=\hbox{%
{$\left[\!\llap{\phantom{%
\begingroup \smaller\smaller\smaller\begin{tabular}{@{}c@{}}%
\phantom{0}\\\phantom{0}\\\phantom{0}
\end{tabular}\endgroup%
}}\right.$}%
\begingroup \smaller\smaller\smaller\begin{tabular}{@{}c@{}}%
-1\\\phantom{0}\\\phantom{0}
\end{tabular}\endgroup%
\kern3pt%
\begingroup \smaller\smaller\smaller\begin{tabular}{@{}c@{}}%
\phantom{0}\\6\\\phantom{0}
\end{tabular}\endgroup%
\kern3pt%
\begingroup \smaller\smaller\smaller\begin{tabular}{@{}c@{}}%
\phantom{0}\\\phantom{0}\\6
\end{tabular}\endgroup%
{$\left.\llap{\phantom{%
\begingroup \smaller\smaller\smaller\begin{tabular}{@{}c@{}}%
\phantom{0}\\\phantom{0}\\\phantom{0}
\end{tabular}\endgroup%
}}\!\right]$}%
{$\left[\!\llap{\phantom{%
\begingroup \smaller\smaller\smaller\begin{tabular}{@{}c@{}}%
0\\0\\0
\end{tabular}\endgroup%
}}\right.$}%
\begingroup \smaller\smaller\smaller\begin{tabular}{@{}c@{}}%
2\\1\\0
\end{tabular}\endgroup%
\kern3pt%
\begingroup \smaller\smaller\smaller\begin{tabular}{@{}c@{}}%
3\\1\\1
\end{tabular}\endgroup%
\kern3pt%
\begingroup \smaller\smaller\smaller\begin{tabular}{@{}c@{}}%
2\\0\\1
\end{tabular}\endgroup%
\kern3pt%
\begingroup \smaller\smaller\smaller\begin{tabular}{@{}c@{}}%
24\\-6\\8
\end{tabular}\endgroup%
\kern3pt%
\begingroup \smaller\smaller\smaller\begin{tabular}{@{}c@{}}%
24\\-8\\6
\end{tabular}\endgroup%
\kern3pt%
\begingroup \smaller\smaller\smaller\begin{tabular}{@{}c@{}}%
12\\-5\\1
\end{tabular}\endgroup%
{$\left.\llap{\phantom{%
\begingroup \smaller\smaller\smaller\begin{tabular}{@{}c@{}}%
0\\0\\0
\end{tabular}\endgroup%
}}\!\right]$}%
}%
\ifdim\wd\matricesbox>\halfwidth\myboxwidth=\hsize\else\myboxwidth=\halfwidth\fi
\vbox{%
\ifdim\myboxwidth=\hsize
\setbox\onelinebox=\hbox{%
\vbox{\hbox{%
$\Pi_{11,2}$ spans $L_{123.8}$%
}\hbox{%
$22|22224\slashtwo422\rtimes D_{2}$%
}%
}%
\hfill\copy\matricesbox
}%
\ifdim\wd\onelinebox>\myboxwidth
\hbox to \myboxwidth{%
$\Pi_{11,2}$ spans $L_{123.8}$%
\hfil
$22|22224\slashtwo422\rtimes D_{2}$%
}%
\box\matricesbox
\else
\hbox to \myboxwidth{%
\unhbox\onelinebox
}%
\fi
\else
\hbox to \myboxwidth{%
$\Pi_{11,2}$ spans $L_{123.8}$%
\hfil}%
\hbox to \myboxwidth{%
$22|22224\slashtwo422\rtimes D_{2}$%
\hfil}%
\box\matricesbox
\fi
}%
\hfill\discretionary{}{}{}%
\setbox\matricesbox=\hbox{%
{$\left[\!\llap{\phantom{%
\begingroup \smaller\smaller\smaller\begin{tabular}{@{}c@{}}%
\phantom{0}\\\phantom{0}\\\phantom{0}
\end{tabular}\endgroup%
}}\right.$}%
\begingroup \smaller\smaller\smaller\begin{tabular}{@{}c@{}}%
-1\\\phantom{0}\\\phantom{0}
\end{tabular}\endgroup%
\kern3pt%
\begingroup \smaller\smaller\smaller\begin{tabular}{@{}c@{}}%
\phantom{0}\\3\\\phantom{0}
\end{tabular}\endgroup%
\kern3pt%
\begingroup \smaller\smaller\smaller\begin{tabular}{@{}c@{}}%
\phantom{0}\\\phantom{0}\\3
\end{tabular}\endgroup%
{$\left.\llap{\phantom{%
\begingroup \smaller\smaller\smaller\begin{tabular}{@{}c@{}}%
\phantom{0}\\\phantom{0}\\\phantom{0}
\end{tabular}\endgroup%
}}\!\right]$}%
{$\left[\!\llap{\phantom{%
\begingroup \smaller\smaller\smaller\begin{tabular}{@{}c@{}}%
0\\0\\0
\end{tabular}\endgroup%
}}\right.$}%
\begingroup \smaller\smaller\smaller\begin{tabular}{@{}c@{}}%
24\\14\\-2
\end{tabular}\endgroup%
\kern3pt%
\begingroup \smaller\smaller\smaller\begin{tabular}{@{}c@{}}%
12\\6\\-4
\end{tabular}\endgroup%
\kern3pt%
\begingroup \smaller\smaller\smaller\begin{tabular}{@{}c@{}}%
12\\4\\-6
\end{tabular}\endgroup%
\kern3pt%
\begingroup \smaller\smaller\smaller\begin{tabular}{@{}c@{}}%
3\\0\\-2
\end{tabular}\endgroup%
\kern3pt%
\begingroup \smaller\smaller\smaller\begin{tabular}{@{}c@{}}%
2\\-1\\-1
\end{tabular}\endgroup%
\kern3pt%
\begingroup \smaller\smaller\smaller\begin{tabular}{@{}c@{}}%
3\\-2\\0
\end{tabular}\endgroup%
{$\left.\llap{\phantom{%
\begingroup \smaller\smaller\smaller\begin{tabular}{@{}c@{}}%
0\\0\\0
\end{tabular}\endgroup%
}}\!\right]$}%
}%
\ifdim\wd\matricesbox>\halfwidth\myboxwidth=\hsize\else\myboxwidth=\halfwidth\fi
\vbox{%
\ifdim\myboxwidth=\hsize
\setbox\onelinebox=\hbox{%
\vbox{\hbox{%
$\Pi_{11,3}$ spans $L_{123.8}$%
}\hbox{%
$224\slashtwo42222|22\rtimes D_{2}$%
}%
}%
\hfill\copy\matricesbox
}%
\ifdim\wd\onelinebox>\myboxwidth
\hbox to \myboxwidth{%
$\Pi_{11,3}$ spans $L_{123.8}$%
\hfil
$224\slashtwo42222|22\rtimes D_{2}$%
}%
\box\matricesbox
\else
\hbox to \myboxwidth{%
\unhbox\onelinebox
}%
\fi
\else
\hbox to \myboxwidth{%
$\Pi_{11,3}$ spans $L_{123.8}$%
\hfil}%
\hbox to \myboxwidth{%
$224\slashtwo42222|22\rtimes D_{2}$%
\hfil}%
\box\matricesbox
\fi
}%
\hfill\discretionary{}{}{}%
\setbox\matricesbox=\hbox{%
{$\left[\!\llap{\phantom{%
\begingroup \smaller\smaller\smaller\begin{tabular}{@{}c@{}}%
\phantom{0}\\\phantom{0}\\\phantom{0}
\end{tabular}\endgroup%
}}\right.$}%
\begingroup \smaller\smaller\smaller\begin{tabular}{@{}c@{}}%
-1\\\phantom{0}\\\phantom{0}
\end{tabular}\endgroup%
\kern3pt%
\begingroup \smaller\smaller\smaller\begin{tabular}{@{}c@{}}%
\phantom{0}\\3\\\phantom{0}
\end{tabular}\endgroup%
\kern3pt%
\begingroup \smaller\smaller\smaller\begin{tabular}{@{}c@{}}%
\phantom{0}\\\phantom{0}\\3
\end{tabular}\endgroup%
{$\left.\llap{\phantom{%
\begingroup \smaller\smaller\smaller\begin{tabular}{@{}c@{}}%
\phantom{0}\\\phantom{0}\\\phantom{0}
\end{tabular}\endgroup%
}}\!\right]$}%
{$\left[\!\llap{\phantom{%
\begingroup \smaller\smaller\smaller\begin{tabular}{@{}c@{}}%
0\\0\\0
\end{tabular}\endgroup%
}}\right.$}%
\begingroup \smaller\smaller\smaller\begin{tabular}{@{}c@{}}%
3\\2\\0
\end{tabular}\endgroup%
\kern3pt%
\begingroup \smaller\smaller\smaller\begin{tabular}{@{}c@{}}%
2\\1\\1
\end{tabular}\endgroup%
\kern3pt%
\begingroup \smaller\smaller\smaller\begin{tabular}{@{}c@{}}%
24\\2\\14
\end{tabular}\endgroup%
\kern3pt%
\begingroup \smaller\smaller\smaller\begin{tabular}{@{}c@{}}%
24\\-2\\14
\end{tabular}\endgroup%
\kern3pt%
\begingroup \smaller\smaller\smaller\begin{tabular}{@{}c@{}}%
2\\-1\\1
\end{tabular}\endgroup%
\kern3pt%
\begingroup \smaller\smaller\smaller\begin{tabular}{@{}c@{}}%
24\\-14\\2
\end{tabular}\endgroup%
{$\left.\llap{\phantom{%
\begingroup \smaller\smaller\smaller\begin{tabular}{@{}c@{}}%
0\\0\\0
\end{tabular}\endgroup%
}}\!\right]$}%
}%
\ifdim\wd\matricesbox>\halfwidth\myboxwidth=\hsize\else\myboxwidth=\halfwidth\fi
\vbox{%
\ifdim\myboxwidth=\hsize
\setbox\onelinebox=\hbox{%
\vbox{\hbox{%
$\Pi_{11,4}$ spans $L_{123.8}$%
}\hbox{%
$2222|22222\slashtwo2\rtimes D_{2}$%
}%
}%
\hfill\copy\matricesbox
}%
\ifdim\wd\onelinebox>\myboxwidth
\hbox to \myboxwidth{%
$\Pi_{11,4}$ spans $L_{123.8}$%
\hfil
$2222|22222\slashtwo2\rtimes D_{2}$%
}%
\box\matricesbox
\else
\hbox to \myboxwidth{%
\unhbox\onelinebox
}%
\fi
\else
\hbox to \myboxwidth{%
$\Pi_{11,4}$ spans $L_{123.8}$%
\hfil}%
\hbox to \myboxwidth{%
$2222|22222\slashtwo2\rtimes D_{2}$%
\hfil}%
\box\matricesbox
\fi
}%
\hfill\discretionary{}{}{}%
\setbox\matricesbox=\hbox{%
{$\left[\!\llap{\phantom{%
\begingroup \smaller\smaller\smaller\begin{tabular}{@{}c@{}}%
\phantom{0}\\\phantom{0}\\\phantom{0}
\end{tabular}\endgroup%
}}\right.$}%
\begingroup \smaller\smaller\smaller\begin{tabular}{@{}c@{}}%
-1\\\phantom{0}\\\phantom{0}
\end{tabular}\endgroup%
\kern3pt%
\begingroup \smaller\smaller\smaller\begin{tabular}{@{}c@{}}%
\phantom{0}\\6\\\phantom{0}
\end{tabular}\endgroup%
\kern3pt%
\begingroup \smaller\smaller\smaller\begin{tabular}{@{}c@{}}%
\phantom{0}\\\phantom{0}\\6
\end{tabular}\endgroup%
{$\left.\llap{\phantom{%
\begingroup \smaller\smaller\smaller\begin{tabular}{@{}c@{}}%
\phantom{0}\\\phantom{0}\\\phantom{0}
\end{tabular}\endgroup%
}}\!\right]$}%
{$\left[\!\llap{\phantom{%
\begingroup \smaller\smaller\smaller\begin{tabular}{@{}c@{}}%
0\\0\\0
\end{tabular}\endgroup%
}}\right.$}%
\begingroup \smaller\smaller\smaller\begin{tabular}{@{}c@{}}%
2\\-1\\0
\end{tabular}\endgroup%
\kern3pt%
\begingroup \smaller\smaller\smaller\begin{tabular}{@{}c@{}}%
24\\-8\\6
\end{tabular}\endgroup%
\kern3pt%
\begingroup \smaller\smaller\smaller\begin{tabular}{@{}c@{}}%
24\\-6\\8
\end{tabular}\endgroup%
\kern3pt%
\begingroup \smaller\smaller\smaller\begin{tabular}{@{}c@{}}%
2\\0\\1
\end{tabular}\endgroup%
\kern3pt%
\begingroup \smaller\smaller\smaller\begin{tabular}{@{}c@{}}%
3\\1\\1
\end{tabular}\endgroup%
\kern3pt%
\begingroup \smaller\smaller\smaller\begin{tabular}{@{}c@{}}%
12\\5\\1
\end{tabular}\endgroup%
{$\left.\llap{\phantom{%
\begingroup \smaller\smaller\smaller\begin{tabular}{@{}c@{}}%
0\\0\\0
\end{tabular}\endgroup%
}}\!\right]$}%
}%
\ifdim\wd\matricesbox>\halfwidth\myboxwidth=\hsize\else\myboxwidth=\halfwidth\fi
\vbox{%
\ifdim\myboxwidth=\hsize
\setbox\onelinebox=\hbox{%
\vbox{\hbox{%
$\Pi_{11,5}$ spans $L_{123.8}$%
}\hbox{%
$|22222\slashtwo22222\rtimes D_{2}$%
}%
}%
\hfill\copy\matricesbox
}%
\ifdim\wd\onelinebox>\myboxwidth
\hbox to \myboxwidth{%
$\Pi_{11,5}$ spans $L_{123.8}$%
\hfil
$|22222\slashtwo22222\rtimes D_{2}$%
}%
\box\matricesbox
\else
\hbox to \myboxwidth{%
\unhbox\onelinebox
}%
\fi
\else
\hbox to \myboxwidth{%
$\Pi_{11,5}$ spans $L_{123.8}$%
\hfil}%
\hbox to \myboxwidth{%
$|22222\slashtwo22222\rtimes D_{2}$%
\hfil}%
\box\matricesbox
\fi
}%
\hfill\discretionary{}{}{}%
\setbox\matricesbox=\hbox{%
{$\left[\!\llap{\phantom{%
\begingroup \smaller\smaller\smaller\begin{tabular}{@{}c@{}}%
\phantom{0}\\\phantom{0}\\\phantom{0}
\end{tabular}\endgroup%
}}\right.$}%
\begingroup \smaller\smaller\smaller\begin{tabular}{@{}c@{}}%
-1\\\phantom{0}\\\phantom{0}
\end{tabular}\endgroup%
\kern3pt%
\begingroup \smaller\smaller\smaller\begin{tabular}{@{}c@{}}%
\phantom{0}\\3\\\phantom{0}
\end{tabular}\endgroup%
\kern3pt%
\begingroup \smaller\smaller\smaller\begin{tabular}{@{}c@{}}%
\phantom{0}\\\phantom{0}\\3
\end{tabular}\endgroup%
{$\left.\llap{\phantom{%
\begingroup \smaller\smaller\smaller\begin{tabular}{@{}c@{}}%
\phantom{0}\\\phantom{0}\\\phantom{0}
\end{tabular}\endgroup%
}}\!\right]$}%
{$\left[\!\llap{\phantom{%
\begingroup \smaller\smaller\smaller\begin{tabular}{@{}c@{}}%
0\\0\\0
\end{tabular}\endgroup%
}}\right.$}%
\begingroup \smaller\smaller\smaller\begin{tabular}{@{}c@{}}%
24\\-14\\-2
\end{tabular}\endgroup%
\kern3pt%
\begingroup \smaller\smaller\smaller\begin{tabular}{@{}c@{}}%
2\\-1\\-1
\end{tabular}\endgroup%
\kern3pt%
\begingroup \smaller\smaller\smaller\begin{tabular}{@{}c@{}}%
3\\0\\-2
\end{tabular}\endgroup%
\kern3pt%
\begingroup \smaller\smaller\smaller\begin{tabular}{@{}c@{}}%
12\\4\\-6
\end{tabular}\endgroup%
\kern3pt%
\begingroup \smaller\smaller\smaller\begin{tabular}{@{}c@{}}%
12\\6\\-4
\end{tabular}\endgroup%
\kern3pt%
\begingroup \smaller\smaller\smaller\begin{tabular}{@{}c@{}}%
3\\2\\0
\end{tabular}\endgroup%
{$\left.\llap{\phantom{%
\begingroup \smaller\smaller\smaller\begin{tabular}{@{}c@{}}%
0\\0\\0
\end{tabular}\endgroup%
}}\!\right]$}%
}%
\ifdim\wd\matricesbox>\halfwidth\myboxwidth=\hsize\else\myboxwidth=\halfwidth\fi
\vbox{%
\ifdim\myboxwidth=\hsize
\setbox\onelinebox=\hbox{%
\vbox{\hbox{%
$\Pi_{11,6}$ spans $L_{123.8}$%
}\hbox{%
$2\slashtwo22222|2222\rtimes D_{2}$%
}%
}%
\hfill\copy\matricesbox
}%
\ifdim\wd\onelinebox>\myboxwidth
\hbox to \myboxwidth{%
$\Pi_{11,6}$ spans $L_{123.8}$%
\hfil
$2\slashtwo22222|2222\rtimes D_{2}$%
}%
\box\matricesbox
\else
\hbox to \myboxwidth{%
\unhbox\onelinebox
}%
\fi
\else
\hbox to \myboxwidth{%
$\Pi_{11,6}$ spans $L_{123.8}$%
\hfil}%
\hbox to \myboxwidth{%
$2\slashtwo22222|2222\rtimes D_{2}$%
\hfil}%
\box\matricesbox
\fi
}%
\hfill\discretionary{}{}{}%
\setbox\matricesbox=\hbox{%
{$\left[\!\llap{\phantom{%
\begingroup \smaller\smaller\smaller\begin{tabular}{@{}c@{}}%
\phantom{0}\\\phantom{0}\\\phantom{0}
\end{tabular}\endgroup%
}}\right.$}%
\begingroup \smaller\smaller\smaller\begin{tabular}{@{}c@{}}%
-1\\\phantom{0}\\\phantom{0}
\end{tabular}\endgroup%
\kern3pt%
\begingroup \smaller\smaller\smaller\begin{tabular}{@{}c@{}}%
\phantom{0}\\6\\\phantom{0}
\end{tabular}\endgroup%
\kern3pt%
\begingroup \smaller\smaller\smaller\begin{tabular}{@{}c@{}}%
\phantom{0}\\\phantom{0}\\6
\end{tabular}\endgroup%
{$\left.\llap{\phantom{%
\begingroup \smaller\smaller\smaller\begin{tabular}{@{}c@{}}%
\phantom{0}\\\phantom{0}\\\phantom{0}
\end{tabular}\endgroup%
}}\!\right]$}%
{$\left[\!\llap{\phantom{%
\begingroup \smaller\smaller\smaller\begin{tabular}{@{}c@{}}%
0\\0\\0
\end{tabular}\endgroup%
}}\right.$}%
\begingroup \smaller\smaller\smaller\begin{tabular}{@{}c@{}}%
2\\1\\0
\end{tabular}\endgroup%
\kern3pt%
\begingroup \smaller\smaller\smaller\begin{tabular}{@{}c@{}}%
3\\1\\1
\end{tabular}\endgroup%
\kern3pt%
\begingroup \smaller\smaller\smaller\begin{tabular}{@{}c@{}}%
12\\1\\5
\end{tabular}\endgroup%
\kern3pt%
\begingroup \smaller\smaller\smaller\begin{tabular}{@{}c@{}}%
12\\-1\\5
\end{tabular}\endgroup%
\kern3pt%
\begingroup \smaller\smaller\smaller\begin{tabular}{@{}c@{}}%
3\\-1\\1
\end{tabular}\endgroup%
\kern3pt%
\begingroup \smaller\smaller\smaller\begin{tabular}{@{}c@{}}%
12\\-5\\1
\end{tabular}\endgroup%
{$\left.\llap{\phantom{%
\begingroup \smaller\smaller\smaller\begin{tabular}{@{}c@{}}%
0\\0\\0
\end{tabular}\endgroup%
}}\!\right]$}%
}%
\ifdim\wd\matricesbox>\halfwidth\myboxwidth=\hsize\else\myboxwidth=\halfwidth\fi
\vbox{%
\ifdim\myboxwidth=\hsize
\setbox\onelinebox=\hbox{%
\vbox{\hbox{%
$\Pi_{11,7}$ spans $L_{123.8}$%
}\hbox{%
$2222|22222\slashtwo2\rtimes D_{2}$%
}%
}%
\hfill\copy\matricesbox
}%
\ifdim\wd\onelinebox>\myboxwidth
\hbox to \myboxwidth{%
$\Pi_{11,7}$ spans $L_{123.8}$%
\hfil
$2222|22222\slashtwo2\rtimes D_{2}$%
}%
\box\matricesbox
\else
\hbox to \myboxwidth{%
\unhbox\onelinebox
}%
\fi
\else
\hbox to \myboxwidth{%
$\Pi_{11,7}$ spans $L_{123.8}$%
\hfil}%
\hbox to \myboxwidth{%
$2222|22222\slashtwo2\rtimes D_{2}$%
\hfil}%
\box\matricesbox
\fi
}%
\hfill\discretionary{}{}{}%
\setbox\matricesbox=\hbox{%
{$\left[\!\llap{\phantom{%
\begingroup \smaller\smaller\smaller\begin{tabular}{@{}c@{}}%
\phantom{0}\\\phantom{0}\\\phantom{0}
\end{tabular}\endgroup%
}}\right.$}%
\begingroup \smaller\smaller\smaller\begin{tabular}{@{}c@{}}%
-3\\\phantom{0}\\\phantom{0}
\end{tabular}\endgroup%
\kern3pt%
\begingroup \smaller\smaller\smaller\begin{tabular}{@{}c@{}}%
\phantom{0}\\1\\\phantom{0}
\end{tabular}\endgroup%
\kern3pt%
\begingroup \smaller\smaller\smaller\begin{tabular}{@{}c@{}}%
\phantom{0}\\\phantom{0}\\3
\end{tabular}\endgroup%
{$\left.\llap{\phantom{%
\begingroup \smaller\smaller\smaller\begin{tabular}{@{}c@{}}%
\phantom{0}\\\phantom{0}\\\phantom{0}
\end{tabular}\endgroup%
}}\!\right]$}%
{$\left[\!\llap{\phantom{%
\begingroup \smaller\smaller\smaller\begin{tabular}{@{}c@{}}%
0\\0\\0
\end{tabular}\endgroup%
}}\right.$}%
\begingroup \smaller\smaller\smaller\begin{tabular}{@{}c@{}}%
1\\-2\\0
\end{tabular}\endgroup%
\kern3pt%
\begingroup \smaller\smaller\smaller\begin{tabular}{@{}c@{}}%
4\\-5\\3
\end{tabular}\endgroup%
\kern3pt%
\begingroup \smaller\smaller\smaller\begin{tabular}{@{}c@{}}%
4\\-2\\4
\end{tabular}\endgroup%
\kern3pt%
\begingroup \smaller\smaller\smaller\begin{tabular}{@{}c@{}}%
4\\2\\4
\end{tabular}\endgroup%
\kern3pt%
\begingroup \smaller\smaller\smaller\begin{tabular}{@{}c@{}}%
4\\5\\3
\end{tabular}\endgroup%
\kern3pt%
\begingroup \smaller\smaller\smaller\begin{tabular}{@{}c@{}}%
4\\7\\1
\end{tabular}\endgroup%
{$\left.\llap{\phantom{%
\begingroup \smaller\smaller\smaller\begin{tabular}{@{}c@{}}%
0\\0\\0
\end{tabular}\endgroup%
}}\!\right]$}%
}%
\ifdim\wd\matricesbox>\halfwidth\myboxwidth=\hsize\else\myboxwidth=\halfwidth\fi
\vbox{%
\ifdim\myboxwidth=\hsize
\setbox\onelinebox=\hbox{%
\vbox{\hbox{%
$\Pi_{11,8}=\hbox{GN}_{57}$ spans $L_{7.7}$%
}\hbox{%
$3\infty|\infty3\infty3\infty\slashthree\infty3\infty\rtimes D_{2}$%
}%
}%
\hfill\copy\matricesbox
}%
\ifdim\wd\onelinebox>\myboxwidth
\hbox to \myboxwidth{%
$\Pi_{11,8}=\hbox{GN}_{57}$ spans $L_{7.7}$%
\hfil
$3\infty|\infty3\infty3\infty\slashthree\infty3\infty\rtimes D_{2}$%
}%
\box\matricesbox
\else
\hbox to \myboxwidth{%
\unhbox\onelinebox
}%
\fi
\else
\hbox to \myboxwidth{%
$\Pi_{11,8}=\hbox{GN}_{57}$ spans $L_{7.7}$%
\hfil}%
\hbox to \myboxwidth{%
$3\infty|\infty3\infty3\infty\slashthree\infty3\infty\rtimes D_{2}$%
\hfil}%
\box\matricesbox
\fi
}%
\hfill\discretionary{}{}{}%
\setbox\matricesbox=\hbox{%
{$\left[\!\llap{\phantom{%
\begingroup \smaller\smaller\smaller\begin{tabular}{@{}c@{}}%
\phantom{0}\\\phantom{0}\\\phantom{0}
\end{tabular}\endgroup%
}}\right.$}%
\begingroup \smaller\smaller\smaller\begin{tabular}{@{}c@{}}%
-4\\\phantom{0}\\\phantom{0}
\end{tabular}\endgroup%
\kern3pt%
\begingroup \smaller\smaller\smaller\begin{tabular}{@{}c@{}}%
\phantom{0}\\1/2\\\phantom{0}
\end{tabular}\endgroup%
\kern3pt%
\begingroup \smaller\smaller\smaller\begin{tabular}{@{}c@{}}%
\phantom{0}\\\phantom{0}\\1/2
\end{tabular}\endgroup%
{$\left.\llap{\phantom{%
\begingroup \smaller\smaller\smaller\begin{tabular}{@{}c@{}}%
\phantom{0}\\\phantom{0}\\\phantom{0}
\end{tabular}\endgroup%
}}\!\right]$}%
{$\left[\!\llap{\phantom{%
\begingroup \smaller\smaller\smaller\begin{tabular}{@{}c@{}}%
0\\0\\0
\end{tabular}\endgroup%
}}\right.$}%
\begingroup \smaller\smaller\smaller\begin{tabular}{@{}c@{}}%
1\\-3\\-1
\end{tabular}\endgroup%
\kern3pt%
\begingroup \smaller\smaller\smaller\begin{tabular}{@{}c@{}}%
4\\-6\\-10
\end{tabular}\endgroup%
\kern3pt%
\begingroup \smaller\smaller\smaller\begin{tabular}{@{}c@{}}%
2\\0\\-6
\end{tabular}\endgroup%
\kern3pt%
\begingroup \smaller\smaller\smaller\begin{tabular}{@{}c@{}}%
4\\6\\-10
\end{tabular}\endgroup%
\kern3pt%
\begingroup \smaller\smaller\smaller\begin{tabular}{@{}c@{}}%
4\\10\\-6
\end{tabular}\endgroup%
\kern3pt%
\begingroup \smaller\smaller\smaller\begin{tabular}{@{}c@{}}%
2\\6\\0
\end{tabular}\endgroup%
{$\left.\llap{\phantom{%
\begingroup \smaller\smaller\smaller\begin{tabular}{@{}c@{}}%
0\\0\\0
\end{tabular}\endgroup%
}}\!\right]$}%
}%
\ifdim\wd\matricesbox>\halfwidth\myboxwidth=\hsize\else\myboxwidth=\halfwidth\fi
\vbox{%
\ifdim\myboxwidth=\hsize
\setbox\onelinebox=\hbox{%
\vbox{\hbox{%
$\Pi_{11,9}$ spans $L_{145.1}$%
}\hbox{%
$4\infty\slashtwo\infty44\infty4|4\infty4\rtimes D_{2}$%
}%
}%
\hfill\copy\matricesbox
}%
\ifdim\wd\onelinebox>\myboxwidth
\hbox to \myboxwidth{%
$\Pi_{11,9}$ spans $L_{145.1}$%
\hfil
$4\infty\slashtwo\infty44\infty4|4\infty4\rtimes D_{2}$%
}%
\box\matricesbox
\else
\hbox to \myboxwidth{%
\unhbox\onelinebox
}%
\fi
\else
\hbox to \myboxwidth{%
$\Pi_{11,9}$ spans $L_{145.1}$%
\hfil}%
\hbox to \myboxwidth{%
$4\infty\slashtwo\infty44\infty4|4\infty4\rtimes D_{2}$%
\hfil}%
\box\matricesbox
\fi
}%
\hfill\discretionary{}{}{}%
\setbox\matricesbox=\hbox{%
{$\left[\!\llap{\phantom{%
\begingroup \smaller\smaller\smaller\begin{tabular}{@{}c@{}}%
\phantom{0}\\\phantom{0}\\\phantom{0}
\end{tabular}\endgroup%
}}\right.$}%
\begingroup \smaller\smaller\smaller\begin{tabular}{@{}c@{}}%
-1\\\phantom{0}\\\phantom{0}
\end{tabular}\endgroup%
\kern3pt%
\begingroup \smaller\smaller\smaller\begin{tabular}{@{}c@{}}%
\phantom{0}\\3/2\\\phantom{0}
\end{tabular}\endgroup%
\kern3pt%
\begingroup \smaller\smaller\smaller\begin{tabular}{@{}c@{}}%
\phantom{0}\\\phantom{0}\\9/2
\end{tabular}\endgroup%
{$\left.\llap{\phantom{%
\begingroup \smaller\smaller\smaller\begin{tabular}{@{}c@{}}%
\phantom{0}\\\phantom{0}\\\phantom{0}
\end{tabular}\endgroup%
}}\!\right]$}%
{$\left[\!\llap{\phantom{%
\begingroup \smaller\smaller\smaller\begin{tabular}{@{}c@{}}%
0\\0\\0
\end{tabular}\endgroup%
}}\right.$}%
\begingroup \smaller\smaller\smaller\begin{tabular}{@{}c@{}}%
2\\-2\\0
\end{tabular}\endgroup%
\kern3pt%
\begingroup \smaller\smaller\smaller\begin{tabular}{@{}c@{}}%
6\\-4\\2
\end{tabular}\endgroup%
\kern3pt%
\begingroup \smaller\smaller\smaller\begin{tabular}{@{}c@{}}%
18\\-6\\8
\end{tabular}\endgroup%
\kern3pt%
\begingroup \smaller\smaller\smaller\begin{tabular}{@{}c@{}}%
6\\1\\3
\end{tabular}\endgroup%
\kern3pt%
\begingroup \smaller\smaller\smaller\begin{tabular}{@{}c@{}}%
18\\9\\7
\end{tabular}\endgroup%
\kern3pt%
\begingroup \smaller\smaller\smaller\begin{tabular}{@{}c@{}}%
6\\5\\1
\end{tabular}\endgroup%
{$\left.\llap{\phantom{%
\begingroup \smaller\smaller\smaller\begin{tabular}{@{}c@{}}%
0\\0\\0
\end{tabular}\endgroup%
}}\!\right]$}%
}%
\ifdim\wd\matricesbox>\halfwidth\myboxwidth=\hsize\else\myboxwidth=\halfwidth\fi
\vbox{%
\ifdim\myboxwidth=\hsize
\setbox\onelinebox=\hbox{%
\vbox{\hbox{%
$\Pi_{11,10}$ spans $L_{155.1}$%
}\hbox{%
$622|22626\slashthree62\rtimes D_{2}$%
}%
}%
\hfill\copy\matricesbox
}%
\ifdim\wd\onelinebox>\myboxwidth
\hbox to \myboxwidth{%
$\Pi_{11,10}$ spans $L_{155.1}$%
\hfil
$622|22626\slashthree62\rtimes D_{2}$%
}%
\box\matricesbox
\else
\hbox to \myboxwidth{%
\unhbox\onelinebox
}%
\fi
\else
\hbox to \myboxwidth{%
$\Pi_{11,10}$ spans $L_{155.1}$%
\hfil}%
\hbox to \myboxwidth{%
$622|22626\slashthree62\rtimes D_{2}$%
\hfil}%
\box\matricesbox
\fi
}%
\hfill\discretionary{}{}{}%
\setbox\matricesbox=\hbox{%
{$\left[\!\llap{\phantom{%
\begingroup \smaller\smaller\smaller
\begin{tabular}{@{}c@{}}%
\phantom{0}\\\phantom{0}\\\phantom{0}
\end{tabular}\endgroup%
}}\right.$}%
\begingroup \smaller\smaller\smaller\begin{tabular}{@{}c@{}}%
-1/8\\\phantom{0}\\\phantom{0}
\end{tabular}\endgroup%
\kern3pt%
\begingroup \smaller\smaller\smaller\begin{tabular}{@{}c@{}}%
\phantom{0}\\77/2\\-7/2
\end{tabular}\endgroup%
\kern3pt%
\begingroup \smaller\smaller\smaller\begin{tabular}{@{}c@{}}%
\phantom{0}\\-7/2\\77/2
\end{tabular}\endgroup%
{$\left.\llap{\phantom{%
\begingroup \smaller\smaller\smaller\begin{tabular}{@{}c@{}}%
\phantom{0}\\\phantom{0}\\\phantom{0}
\end{tabular}\endgroup%
}}\!\right]$}%
{$\left[\!\llap{\phantom{%
\begingroup \smaller\smaller\smaller\begin{tabular}{@{}c@{}}%
0\\0\\0
\end{tabular}\endgroup%
}}\right.$}%
\begingroup \smaller\smaller\smaller\begin{tabular}{@{}c@{}}%
20\\-1\\-1
\end{tabular}\endgroup%
\kern3pt%
\begingroup \smaller\smaller\smaller\begin{tabular}{@{}c@{}}%
84\\-1\\-5
\end{tabular}\endgroup%
\kern3pt%
\begingroup \smaller\smaller\smaller\begin{tabular}{@{}c@{}}%
70\\1\\-4
\end{tabular}\endgroup%
\kern3pt%
\begingroup \smaller\smaller\smaller\begin{tabular}{@{}c@{}}%
48\\2\\-2
\end{tabular}\endgroup%
\kern3pt%
\begingroup \smaller\smaller\smaller\begin{tabular}{@{}c@{}}%
14\\1\\0
\end{tabular}\endgroup%
\kern3pt%
\begingroup \smaller\smaller\smaller\begin{tabular}{@{}c@{}}%
20\\1\\1
\end{tabular}\endgroup%
\kern3pt%
\begingroup \smaller\smaller\smaller\begin{tabular}{@{}c@{}}%
84\\1\\5
\end{tabular}\endgroup%
\kern3pt%
\begingroup \smaller\smaller\smaller\begin{tabular}{@{}c@{}}%
70\\-1\\4
\end{tabular}\endgroup%
\kern3pt%
\begingroup \smaller\smaller\smaller\begin{tabular}{@{}c@{}}%
48\\-2\\2
\end{tabular}\endgroup%
\kern3pt%
\begingroup \smaller\smaller\smaller\begin{tabular}{@{}c@{}}%
70\\-4\\1
\end{tabular}\endgroup%
\kern3pt%
\begingroup \smaller\smaller\smaller\begin{tabular}{@{}c@{}}%
84\\-5\\-1
\end{tabular}\endgroup%
{$\left.\llap{\phantom{%
\begingroup \smaller\smaller\smaller\begin{tabular}{@{}c@{}}%
0\\0\\0
\end{tabular}\endgroup%
}}\!\right]$}%
}%
\ifdim\wd\matricesbox>\halfwidth\myboxwidth=\hsize\else\myboxwidth=\halfwidth\fi
\vbox{%
\ifdim\myboxwidth=\hsize
\setbox\onelinebox=\hbox{%
\vbox{\hbox{%
$\Pi_{11,11}$ spans $L_{69.11}$%
}\hbox{%
$22222222222$%
}%
}%
\hfill\copy\matricesbox
}%
\ifdim\wd\onelinebox>\myboxwidth
\hbox to \myboxwidth{%
$\Pi_{11,11}$ spans $L_{69.11}$%
\hfil
$22222222222$%
}%
\box\matricesbox
\else
\hbox to \myboxwidth{%
\unhbox\onelinebox
}%
\fi
\else
\hbox to \myboxwidth{%
$\Pi_{11,11}$ spans $L_{69.11}$%
\hfil}%
\hbox to \myboxwidth{%
$22222222222$%
\hfil}%
\box\matricesbox
\fi
}%
\hfill\discretionary{}{}{}%
\setbox\matricesbox=\hbox{%
{$\left[\!\llap{\phantom{%
\begingroup \smaller\smaller\smaller
\begin{tabular}{@{}c@{}}%
\phantom{0}\\\phantom{0}\\\phantom{0}
\end{tabular}\endgroup%
}}\right.$}%
\begingroup \smaller\smaller\smaller\begin{tabular}{@{}c@{}}%
-1/2\\\phantom{0}\\\phantom{0}
\end{tabular}\endgroup%
\kern3pt%
\begingroup \smaller\smaller\smaller\begin{tabular}{@{}c@{}}%
\phantom{0}\\15/2\\\phantom{0}
\end{tabular}\endgroup%
\kern3pt%
\begingroup \smaller\smaller\smaller\begin{tabular}{@{}c@{}}%
\phantom{0}\\\phantom{0}\\10
\end{tabular}\endgroup%
{$\left.\llap{\phantom{%
\begingroup \smaller\smaller\smaller\begin{tabular}{@{}c@{}}%
\phantom{0}\\\phantom{0}\\\phantom{0}
\end{tabular}\endgroup%
}}\!\right]$}%
{$\left[\!\llap{\phantom{%
\begingroup \smaller\smaller\smaller\begin{tabular}{@{}c@{}}%
0\\0\\0
\end{tabular}\endgroup%
}}\right.$}%
\begingroup \smaller\smaller\smaller\begin{tabular}{@{}c@{}}%
3\\-1\\0
\end{tabular}\endgroup%
\kern3pt%
\begingroup \smaller\smaller\smaller\begin{tabular}{@{}c@{}}%
5\\-1\\-1
\end{tabular}\endgroup%
\kern3pt%
\begingroup \smaller\smaller\smaller\begin{tabular}{@{}c@{}}%
8\\0\\-2
\end{tabular}\endgroup%
\kern3pt%
\begingroup \smaller\smaller\smaller\begin{tabular}{@{}c@{}}%
30\\4\\-6
\end{tabular}\endgroup%
\kern3pt%
\begingroup \smaller\smaller\smaller\begin{tabular}{@{}c@{}}%
40\\8\\-6
\end{tabular}\endgroup%
\kern3pt%
\begingroup \smaller\smaller\smaller\begin{tabular}{@{}c@{}}%
3\\1\\0
\end{tabular}\endgroup%
\kern3pt%
\begingroup \smaller\smaller\smaller\begin{tabular}{@{}c@{}}%
40\\8\\6
\end{tabular}\endgroup%
\kern3pt%
\begingroup \smaller\smaller\smaller\begin{tabular}{@{}c@{}}%
30\\4\\6
\end{tabular}\endgroup%
\kern3pt%
\begingroup \smaller\smaller\smaller\begin{tabular}{@{}c@{}}%
8\\0\\2
\end{tabular}\endgroup%
\kern3pt%
\begingroup \smaller\smaller\smaller\begin{tabular}{@{}c@{}}%
30\\-4\\6
\end{tabular}\endgroup%
\kern3pt%
\begingroup \smaller\smaller\smaller\begin{tabular}{@{}c@{}}%
40\\-8\\6
\end{tabular}\endgroup%
{$\left.\llap{\phantom{%
\begingroup \smaller\smaller\smaller\begin{tabular}{@{}c@{}}%
0\\0\\0
\end{tabular}\endgroup%
}}\!\right]$}%
}%
\ifdim\wd\matricesbox>\halfwidth\myboxwidth=\hsize\else\myboxwidth=\halfwidth\fi
\vbox{%
\ifdim\myboxwidth=\hsize
\setbox\onelinebox=\hbox{%
\vbox{\hbox{%
$\Pi_{11,12}$ spans $L_{17.11}$%
}\hbox{%
$22222222222$%
}%
}%
\hfill\copy\matricesbox
}%
\ifdim\wd\onelinebox>\myboxwidth
\hbox to \myboxwidth{%
$\Pi_{11,12}$ spans $L_{17.11}$%
\hfil
$22222222222$%
}%
\box\matricesbox
\else
\hbox to \myboxwidth{%
\unhbox\onelinebox
}%
\fi
\else
\hbox to \myboxwidth{%
$\Pi_{11,12}$ spans $L_{17.11}$%
\hfil}%
\hbox to \myboxwidth{%
$22222222222$%
\hfil}%
\box\matricesbox
\fi
}%
\hfill\discretionary{}{}{}%
\setbox\matricesbox=\hbox{%
{$\left[\!\llap{\phantom{%
\begingroup \smaller\smaller\smaller
\begin{tabular}{@{}c@{}}%
\phantom{0}\\\phantom{0}\\\phantom{0}
\end{tabular}\endgroup%
}}\right.$}%
\begingroup \smaller\smaller\smaller\begin{tabular}{@{}c@{}}%
-1\\\phantom{0}\\\phantom{0}
\end{tabular}\endgroup%
\kern3pt%
\begingroup \smaller\smaller\smaller\begin{tabular}{@{}c@{}}%
\phantom{0}\\6\\\phantom{0}
\end{tabular}\endgroup%
\kern3pt%
\begingroup \smaller\smaller\smaller\begin{tabular}{@{}c@{}}%
\phantom{0}\\\phantom{0}\\6
\end{tabular}\endgroup%
{$\left.\llap{\phantom{%
\begingroup \smaller\smaller\smaller\begin{tabular}{@{}c@{}}%
\phantom{0}\\\phantom{0}\\\phantom{0}
\end{tabular}\endgroup%
}}\!\right]$}%
{$\left[\!\llap{\phantom{%
\begingroup \smaller\smaller\smaller\begin{tabular}{@{}c@{}}%
0\\0\\0
\end{tabular}\endgroup%
}}\right.$}%
\begingroup \smaller\smaller\smaller\begin{tabular}{@{}c@{}}%
2\\1\\0
\end{tabular}\endgroup%
\kern3pt%
\begingroup \smaller\smaller\smaller\begin{tabular}{@{}c@{}}%
24\\8\\6
\end{tabular}\endgroup%
\kern3pt%
\begingroup \smaller\smaller\smaller\begin{tabular}{@{}c@{}}%
24\\6\\8
\end{tabular}\endgroup%
\kern3pt%
\begingroup \smaller\smaller\smaller\begin{tabular}{@{}c@{}}%
12\\1\\5
\end{tabular}\endgroup%
\kern3pt%
\begingroup \smaller\smaller\smaller\begin{tabular}{@{}c@{}}%
12\\-1\\5
\end{tabular}\endgroup%
\kern3pt%
\begingroup \smaller\smaller\smaller\begin{tabular}{@{}c@{}}%
3\\-1\\1
\end{tabular}\endgroup%
\kern3pt%
\begingroup \smaller\smaller\smaller\begin{tabular}{@{}c@{}}%
2\\-1\\0
\end{tabular}\endgroup%
\kern3pt%
\begingroup \smaller\smaller\smaller\begin{tabular}{@{}c@{}}%
3\\-1\\-1
\end{tabular}\endgroup%
\kern3pt%
\begingroup \smaller\smaller\smaller\begin{tabular}{@{}c@{}}%
2\\0\\-1
\end{tabular}\endgroup%
\kern3pt%
\begingroup \smaller\smaller\smaller\begin{tabular}{@{}c@{}}%
24\\6\\-8
\end{tabular}\endgroup%
\kern3pt%
\begingroup \smaller\smaller\smaller\begin{tabular}{@{}c@{}}%
24\\8\\-6
\end{tabular}\endgroup%
{$\left.\llap{\phantom{%
\begingroup \smaller\smaller\smaller\begin{tabular}{@{}c@{}}%
0\\0\\0
\end{tabular}\endgroup%
}}\!\right]$}%
}%
\ifdim\wd\matricesbox>\halfwidth\myboxwidth=\hsize\else\myboxwidth=\halfwidth\fi
\vbox{%
\ifdim\myboxwidth=\hsize
\setbox\onelinebox=\hbox{%
\vbox{\hbox{%
$\Pi_{11,13}$ spans $L_{123.8}$%
}\hbox{%
$22422222222$%
}%
}%
\hfill\copy\matricesbox
}%
\ifdim\wd\onelinebox>\myboxwidth
\hbox to \myboxwidth{%
$\Pi_{11,13}$ spans $L_{123.8}$%
\hfil
$22422222222$%
}%
\box\matricesbox
\else
\hbox to \myboxwidth{%
\unhbox\onelinebox
}%
\fi
\else
\hbox to \myboxwidth{%
$\Pi_{11,13}$ spans $L_{123.8}$%
\hfil}%
\hbox to \myboxwidth{%
$22422222222$%
\hfil}%
\box\matricesbox
\fi
}%
\hfill\discretionary{}{}{}%
\setbox\matricesbox=\hbox{%
{$\left[\!\llap{\phantom{%
\begingroup \smaller\smaller\smaller
\begin{tabular}{@{}c@{}}%
\phantom{0}\\\phantom{0}\\\phantom{0}
\end{tabular}\endgroup%
}}\right.$}%
\begingroup \smaller\smaller\smaller\begin{tabular}{@{}c@{}}%
-1\\\phantom{0}\\\phantom{0}
\end{tabular}\endgroup%
\kern3pt%
\begingroup \smaller\smaller\smaller\begin{tabular}{@{}c@{}}%
\phantom{0}\\6\\\phantom{0}
\end{tabular}\endgroup%
\kern3pt%
\begingroup \smaller\smaller\smaller\begin{tabular}{@{}c@{}}%
\phantom{0}\\\phantom{0}\\6
\end{tabular}\endgroup%
{$\left.\llap{\phantom{%
\begingroup \smaller\smaller\smaller\begin{tabular}{@{}c@{}}%
\phantom{0}\\\phantom{0}\\\phantom{0}
\end{tabular}\endgroup%
}}\!\right]$}%
{$\left[\!\llap{\phantom{%
\begingroup \smaller\smaller\smaller\begin{tabular}{@{}c@{}}%
0\\0\\0
\end{tabular}\endgroup%
}}\right.$}%
\begingroup \smaller\smaller\smaller\begin{tabular}{@{}c@{}}%
2\\1\\0
\end{tabular}\endgroup%
\kern3pt%
\begingroup \smaller\smaller\smaller\begin{tabular}{@{}c@{}}%
24\\8\\6
\end{tabular}\endgroup%
\kern3pt%
\begingroup \smaller\smaller\smaller\begin{tabular}{@{}c@{}}%
24\\6\\8
\end{tabular}\endgroup%
\kern3pt%
\begingroup \smaller\smaller\smaller\begin{tabular}{@{}c@{}}%
12\\1\\5
\end{tabular}\endgroup%
\kern3pt%
\begingroup \smaller\smaller\smaller\begin{tabular}{@{}c@{}}%
12\\-1\\5
\end{tabular}\endgroup%
\kern3pt%
\begingroup \smaller\smaller\smaller\begin{tabular}{@{}c@{}}%
3\\-1\\1
\end{tabular}\endgroup%
\kern3pt%
\begingroup \smaller\smaller\smaller\begin{tabular}{@{}c@{}}%
2\\-1\\0
\end{tabular}\endgroup%
\kern3pt%
\begingroup \smaller\smaller\smaller\begin{tabular}{@{}c@{}}%
3\\-1\\-1
\end{tabular}\endgroup%
\kern3pt%
\begingroup \smaller\smaller\smaller\begin{tabular}{@{}c@{}}%
12\\-1\\-5
\end{tabular}\endgroup%
\kern3pt%
\begingroup \smaller\smaller\smaller\begin{tabular}{@{}c@{}}%
12\\1\\-5
\end{tabular}\endgroup%
\kern3pt%
\begingroup \smaller\smaller\smaller\begin{tabular}{@{}c@{}}%
3\\1\\-1
\end{tabular}\endgroup%
{$\left.\llap{\phantom{%
\begingroup \smaller\smaller\smaller\begin{tabular}{@{}c@{}}%
0\\0\\0
\end{tabular}\endgroup%
}}\!\right]$}%
}%
\ifdim\wd\matricesbox>\halfwidth\myboxwidth=\hsize\else\myboxwidth=\halfwidth\fi
\vbox{%
\ifdim\myboxwidth=\hsize
\setbox\onelinebox=\hbox{%
\vbox{\hbox{%
$\Pi_{11,14}$ spans $L_{123.8}$%
}\hbox{%
$22422222222$%
}%
}%
\hfill\copy\matricesbox
}%
\ifdim\wd\onelinebox>\myboxwidth
\hbox to \myboxwidth{%
$\Pi_{11,14}$ spans $L_{123.8}$%
\hfil
$22422222222$%
}%
\box\matricesbox
\else
\hbox to \myboxwidth{%
\unhbox\onelinebox
}%
\fi
\else
\hbox to \myboxwidth{%
$\Pi_{11,14}$ spans $L_{123.8}$%
\hfil}%
\hbox to \myboxwidth{%
$22422222222$%
\hfil}%
\box\matricesbox
\fi
}%
\hfill\discretionary{}{}{}%
\setbox\matricesbox=\hbox{%
{$\left[\!\llap{\phantom{%
\begingroup \smaller\smaller\smaller
\begin{tabular}{@{}c@{}}%
\phantom{0}\\\phantom{0}\\\phantom{0}
\end{tabular}\endgroup%
}}\right.$}%
\begingroup \smaller\smaller\smaller\begin{tabular}{@{}c@{}}%
-1\\\phantom{0}\\\phantom{0}
\end{tabular}\endgroup%
\kern3pt%
\begingroup \smaller\smaller\smaller\begin{tabular}{@{}c@{}}%
\phantom{0}\\6\\\phantom{0}
\end{tabular}\endgroup%
\kern3pt%
\begingroup \smaller\smaller\smaller\begin{tabular}{@{}c@{}}%
\phantom{0}\\\phantom{0}\\6
\end{tabular}\endgroup%
{$\left.\llap{\phantom{%
\begingroup \smaller\smaller\smaller\begin{tabular}{@{}c@{}}%
\phantom{0}\\\phantom{0}\\\phantom{0}
\end{tabular}\endgroup%
}}\!\right]$}%
{$\left[\!\llap{\phantom{%
\begingroup \smaller\smaller\smaller\begin{tabular}{@{}c@{}}%
0\\0\\0
\end{tabular}\endgroup%
}}\right.$}%
\begingroup \smaller\smaller\smaller\begin{tabular}{@{}c@{}}%
2\\1\\0
\end{tabular}\endgroup%
\kern3pt%
\begingroup \smaller\smaller\smaller\begin{tabular}{@{}c@{}}%
24\\8\\6
\end{tabular}\endgroup%
\kern3pt%
\begingroup \smaller\smaller\smaller\begin{tabular}{@{}c@{}}%
24\\6\\8
\end{tabular}\endgroup%
\kern3pt%
\begingroup \smaller\smaller\smaller\begin{tabular}{@{}c@{}}%
12\\1\\5
\end{tabular}\endgroup%
\kern3pt%
\begingroup \smaller\smaller\smaller\begin{tabular}{@{}c@{}}%
12\\-1\\5
\end{tabular}\endgroup%
\kern3pt%
\begingroup \smaller\smaller\smaller\begin{tabular}{@{}c@{}}%
3\\-1\\1
\end{tabular}\endgroup%
\kern3pt%
\begingroup \smaller\smaller\smaller\begin{tabular}{@{}c@{}}%
2\\-1\\0
\end{tabular}\endgroup%
\kern3pt%
\begingroup \smaller\smaller\smaller\begin{tabular}{@{}c@{}}%
24\\-8\\-6
\end{tabular}\endgroup%
\kern3pt%
\begingroup \smaller\smaller\smaller\begin{tabular}{@{}c@{}}%
24\\-6\\-8
\end{tabular}\endgroup%
\kern3pt%
\begingroup \smaller\smaller\smaller\begin{tabular}{@{}c@{}}%
2\\0\\-1
\end{tabular}\endgroup%
\kern3pt%
\begingroup \smaller\smaller\smaller\begin{tabular}{@{}c@{}}%
3\\1\\-1
\end{tabular}\endgroup%
{$\left.\llap{\phantom{%
\begingroup \smaller\smaller\smaller\begin{tabular}{@{}c@{}}%
0\\0\\0
\end{tabular}\endgroup%
}}\!\right]$}%
}%
\ifdim\wd\matricesbox>\halfwidth\myboxwidth=\hsize\else\myboxwidth=\halfwidth\fi
\vbox{%
\ifdim\myboxwidth=\hsize
\setbox\onelinebox=\hbox{%
\vbox{\hbox{%
$\Pi_{11,15}$ spans $L_{123.8}$%
}\hbox{%
$22422222222$%
}%
}%
\hfill\copy\matricesbox
}%
\ifdim\wd\onelinebox>\myboxwidth
\hbox to \myboxwidth{%
$\Pi_{11,15}$ spans $L_{123.8}$%
\hfil
$22422222222$%
}%
\box\matricesbox
\else
\hbox to \myboxwidth{%
\unhbox\onelinebox
}%
\fi
\else
\hbox to \myboxwidth{%
$\Pi_{11,15}$ spans $L_{123.8}$%
\hfil}%
\hbox to \myboxwidth{%
$22422222222$%
\hfil}%
\box\matricesbox
\fi
}%
\hfill\discretionary{}{}{}%
\setbox\matricesbox=\hbox{%
{$\left[\!\llap{\phantom{%
\begingroup \smaller\smaller\smaller
\begin{tabular}{@{}c@{}}%
\phantom{0}\\\phantom{0}\\\phantom{0}
\end{tabular}\endgroup%
}}\right.$}%
\begingroup \smaller\smaller\smaller\begin{tabular}{@{}c@{}}%
-1\\\phantom{0}\\\phantom{0}
\end{tabular}\endgroup%
\kern3pt%
\begingroup \smaller\smaller\smaller\begin{tabular}{@{}c@{}}%
\phantom{0}\\6\\\phantom{0}
\end{tabular}\endgroup%
\kern3pt%
\begingroup \smaller\smaller\smaller\begin{tabular}{@{}c@{}}%
\phantom{0}\\\phantom{0}\\6
\end{tabular}\endgroup%
{$\left.\llap{\phantom{%
\begingroup \smaller\smaller\smaller\begin{tabular}{@{}c@{}}%
\phantom{0}\\\phantom{0}\\\phantom{0}
\end{tabular}\endgroup%
}}\!\right]$}%
{$\left[\!\llap{\phantom{%
\begingroup \smaller\smaller\smaller\begin{tabular}{@{}c@{}}%
0\\0\\0
\end{tabular}\endgroup%
}}\right.$}%
\begingroup \smaller\smaller\smaller\begin{tabular}{@{}c@{}}%
2\\1\\0
\end{tabular}\endgroup%
\kern3pt%
\begingroup \smaller\smaller\smaller\begin{tabular}{@{}c@{}}%
24\\8\\6
\end{tabular}\endgroup%
\kern3pt%
\begingroup \smaller\smaller\smaller\begin{tabular}{@{}c@{}}%
24\\6\\8
\end{tabular}\endgroup%
\kern3pt%
\begingroup \smaller\smaller\smaller\begin{tabular}{@{}c@{}}%
12\\1\\5
\end{tabular}\endgroup%
\kern3pt%
\begingroup \smaller\smaller\smaller\begin{tabular}{@{}c@{}}%
12\\-1\\5
\end{tabular}\endgroup%
\kern3pt%
\begingroup \smaller\smaller\smaller\begin{tabular}{@{}c@{}}%
3\\-1\\1
\end{tabular}\endgroup%
\kern3pt%
\begingroup \smaller\smaller\smaller\begin{tabular}{@{}c@{}}%
12\\-5\\1
\end{tabular}\endgroup%
\kern3pt%
\begingroup \smaller\smaller\smaller\begin{tabular}{@{}c@{}}%
12\\-5\\-1
\end{tabular}\endgroup%
\kern3pt%
\begingroup \smaller\smaller\smaller\begin{tabular}{@{}c@{}}%
3\\-1\\-1
\end{tabular}\endgroup%
\kern3pt%
\begingroup \smaller\smaller\smaller\begin{tabular}{@{}c@{}}%
2\\0\\-1
\end{tabular}\endgroup%
\kern3pt%
\begingroup \smaller\smaller\smaller\begin{tabular}{@{}c@{}}%
3\\1\\-1
\end{tabular}\endgroup%
{$\left.\llap{\phantom{%
\begingroup \smaller\smaller\smaller\begin{tabular}{@{}c@{}}%
0\\0\\0
\end{tabular}\endgroup%
}}\!\right]$}%
}%
\ifdim\wd\matricesbox>\halfwidth\myboxwidth=\hsize\else\myboxwidth=\halfwidth\fi
\vbox{%
\ifdim\myboxwidth=\hsize
\setbox\onelinebox=\hbox{%
\vbox{\hbox{%
$\Pi_{11,16}$ spans $L_{123.8}$%
}\hbox{%
$22422222222$%
}%
}%
\hfill\copy\matricesbox
}%
\ifdim\wd\onelinebox>\myboxwidth
\hbox to \myboxwidth{%
$\Pi_{11,16}$ spans $L_{123.8}$%
\hfil
$22422222222$%
}%
\box\matricesbox
\else
\hbox to \myboxwidth{%
\unhbox\onelinebox
}%
\fi
\else
\hbox to \myboxwidth{%
$\Pi_{11,16}$ spans $L_{123.8}$%
\hfil}%
\hbox to \myboxwidth{%
$22422222222$%
\hfil}%
\box\matricesbox
\fi
}%
\hfill\discretionary{}{}{}%
\setbox\matricesbox=\hbox{%
{$\left[\!\llap{\phantom{%
\begingroup \smaller\smaller\smaller
\begin{tabular}{@{}c@{}}%
\phantom{0}\\\phantom{0}\\\phantom{0}
\end{tabular}\endgroup%
}}\right.$}%
\begingroup \smaller\smaller\smaller\begin{tabular}{@{}c@{}}%
-1\\\phantom{0}\\\phantom{0}
\end{tabular}\endgroup%
\kern3pt%
\begingroup \smaller\smaller\smaller\begin{tabular}{@{}c@{}}%
\phantom{0}\\6\\-3
\end{tabular}\endgroup%
\kern3pt%
\begingroup \smaller\smaller\smaller\begin{tabular}{@{}c@{}}%
\phantom{0}\\-3\\6
\end{tabular}\endgroup%
{$\left.\llap{\phantom{%
\begingroup \smaller\smaller\smaller\begin{tabular}{@{}c@{}}%
\phantom{0}\\\phantom{0}\\\phantom{0}
\end{tabular}\endgroup%
}}\!\right]$}%
{$\left[\!\llap{\phantom{%
\begingroup \smaller\smaller\smaller\begin{tabular}{@{}c@{}}%
0\\0\\0
\end{tabular}\endgroup%
}}\right.$}%
\begingroup \smaller\smaller\smaller\begin{tabular}{@{}c@{}}%
6\\-1\\2
\end{tabular}\endgroup%
\kern3pt%
\begingroup \smaller\smaller\smaller\begin{tabular}{@{}c@{}}%
6\\1\\3
\end{tabular}\endgroup%
\kern3pt%
\begingroup \smaller\smaller\smaller\begin{tabular}{@{}c@{}}%
2\\1\\1
\end{tabular}\endgroup%
\kern3pt%
\begingroup \smaller\smaller\smaller\begin{tabular}{@{}c@{}}%
6\\3\\1
\end{tabular}\endgroup%
\kern3pt%
\begingroup \smaller\smaller\smaller\begin{tabular}{@{}c@{}}%
18\\7\\-1
\end{tabular}\endgroup%
\kern3pt%
\begingroup \smaller\smaller\smaller\begin{tabular}{@{}c@{}}%
6\\1\\-2
\end{tabular}\endgroup%
\kern3pt%
\begingroup \smaller\smaller\smaller\begin{tabular}{@{}c@{}}%
18\\-1\\-8
\end{tabular}\endgroup%
\kern3pt%
\begingroup \smaller\smaller\smaller\begin{tabular}{@{}c@{}}%
6\\-2\\-3
\end{tabular}\endgroup%
\kern3pt%
\begingroup \smaller\smaller\smaller\begin{tabular}{@{}c@{}}%
18\\-8\\-7
\end{tabular}\endgroup%
\kern3pt%
\begingroup \smaller\smaller\smaller\begin{tabular}{@{}c@{}}%
6\\-3\\-1
\end{tabular}\endgroup%
\kern3pt%
\begingroup \smaller\smaller\smaller\begin{tabular}{@{}c@{}}%
18\\-7\\1
\end{tabular}\endgroup%
{$\left.\llap{\phantom{%
\begingroup \smaller\smaller\smaller\begin{tabular}{@{}c@{}}%
0\\0\\0
\end{tabular}\endgroup%
}}\!\right]$}%
}%
\ifdim\wd\matricesbox>\halfwidth\myboxwidth=\hsize\else\myboxwidth=\halfwidth\fi
\vbox{%
\ifdim\myboxwidth=\hsize
\setbox\onelinebox=\hbox{%
\vbox{\hbox{%
$\Pi_{11,17}$ spans $L_{155.1}$%
}\hbox{%
$32226262626$%
}%
}%
\hfill\copy\matricesbox
}%
\ifdim\wd\onelinebox>\myboxwidth
\hbox to \myboxwidth{%
$\Pi_{11,17}$ spans $L_{155.1}$%
\hfil
$32226262626$%
}%
\box\matricesbox
\else
\hbox to \myboxwidth{%
\unhbox\onelinebox
}%
\fi
\else
\hbox to \myboxwidth{%
$\Pi_{11,17}$ spans $L_{155.1}$%
\hfil}%
\hbox to \myboxwidth{%
$32226262626$%
\hfil}%
\box\matricesbox
\fi
}%
\hfill\discretionary{}{}{}%
\setbox\matricesbox=\hbox{%
{$\left[\!\llap{\phantom{%
\begingroup \smaller\smaller\smaller
\begin{tabular}{@{}c@{}}%
\phantom{0}\\\phantom{0}\\\phantom{0}
\end{tabular}\endgroup%
}}\right.$}%
\begingroup \smaller\smaller\smaller\begin{tabular}{@{}c@{}}%
-1\\\phantom{0}\\\phantom{0}
\end{tabular}\endgroup%
\kern3pt%
\begingroup \smaller\smaller\smaller\begin{tabular}{@{}c@{}}%
\phantom{0}\\6\\-3
\end{tabular}\endgroup%
\kern3pt%
\begingroup \smaller\smaller\smaller\begin{tabular}{@{}c@{}}%
\phantom{0}\\-3\\6
\end{tabular}\endgroup%
{$\left.\llap{\phantom{%
\begingroup \smaller\smaller\smaller\begin{tabular}{@{}c@{}}%
\phantom{0}\\\phantom{0}\\\phantom{0}
\end{tabular}\endgroup%
}}\!\right]$}%
{$\left[\!\llap{\phantom{%
\begingroup \smaller\smaller\smaller\begin{tabular}{@{}c@{}}%
0\\0\\0
\end{tabular}\endgroup%
}}\right.$}%
\begingroup \smaller\smaller\smaller\begin{tabular}{@{}c@{}}%
6\\-1\\-3
\end{tabular}\endgroup%
\kern3pt%
\begingroup \smaller\smaller\smaller\begin{tabular}{@{}c@{}}%
18\\-7\\-8
\end{tabular}\endgroup%
\kern3pt%
\begingroup \smaller\smaller\smaller\begin{tabular}{@{}c@{}}%
6\\-3\\-2
\end{tabular}\endgroup%
\kern3pt%
\begingroup \smaller\smaller\smaller\begin{tabular}{@{}c@{}}%
2\\-1\\0
\end{tabular}\endgroup%
\kern3pt%
\begingroup \smaller\smaller\smaller\begin{tabular}{@{}c@{}}%
6\\-1\\2
\end{tabular}\endgroup%
\kern3pt%
\begingroup \smaller\smaller\smaller\begin{tabular}{@{}c@{}}%
18\\1\\8
\end{tabular}\endgroup%
\kern3pt%
\begingroup \smaller\smaller\smaller\begin{tabular}{@{}c@{}}%
6\\2\\3
\end{tabular}\endgroup%
\kern3pt%
\begingroup \smaller\smaller\smaller\begin{tabular}{@{}c@{}}%
6\\3\\2
\end{tabular}\endgroup%
\kern3pt%
\begingroup \smaller\smaller\smaller\begin{tabular}{@{}c@{}}%
18\\8\\1
\end{tabular}\endgroup%
\kern3pt%
\begingroup \smaller\smaller\smaller\begin{tabular}{@{}c@{}}%
6\\2\\-1
\end{tabular}\endgroup%
\kern3pt%
\begingroup \smaller\smaller\smaller\begin{tabular}{@{}c@{}}%
18\\1\\-7
\end{tabular}\endgroup%
{$\left.\llap{\phantom{%
\begingroup \smaller\smaller\smaller\begin{tabular}{@{}c@{}}%
0\\0\\0
\end{tabular}\endgroup%
}}\!\right]$}%
}%
\ifdim\wd\matricesbox>\halfwidth\myboxwidth=\hsize\else\myboxwidth=\halfwidth\fi
\vbox{%
\ifdim\myboxwidth=\hsize
\setbox\onelinebox=\hbox{%
\vbox{\hbox{%
$\Pi_{11,18}$ spans $L_{155.1}$%
}\hbox{%
$62222636262$%
}%
}%
\hfill\copy\matricesbox
}%
\ifdim\wd\onelinebox>\myboxwidth
\hbox to \myboxwidth{%
$\Pi_{11,18}$ spans $L_{155.1}$%
\hfil
$62222636262$%
}%
\box\matricesbox
\else
\hbox to \myboxwidth{%
\unhbox\onelinebox
}%
\fi
\else
\hbox to \myboxwidth{%
$\Pi_{11,18}$ spans $L_{155.1}$%
\hfil}%
\hbox to \myboxwidth{%
$62222636262$%
\hfil}%
\box\matricesbox
\fi
}%
\hfill\discretionary{}{}{}%

\vskip2pt\hrule\vskip2pt

\leavevmode\setbox\matricesbox=\hbox{%
{$\left[\!\llap{\phantom{%
\begingroup \smaller\smaller\smaller\begin{tabular}{@{}c@{}}%
\phantom{0}\\\phantom{0}\\\phantom{0}\\\phantom{0}
\end{tabular}\endgroup%
}}\right.$}%
\begingroup \smaller\smaller\smaller\begin{tabular}{@{}c@{}}%
-1\\\phantom{0}\\\phantom{0}\\\phantom{0}
\end{tabular}\endgroup%
\kern3pt%
\begingroup \smaller\smaller\smaller\begin{tabular}{@{}c@{}}%
\phantom{0}\\10\\\phantom{0}\\\phantom{0}
\end{tabular}\endgroup%
\kern3pt%
\begingroup \smaller\smaller\smaller\begin{tabular}{@{}c@{}}%
\phantom{0}\\\phantom{0}\\10\\\phantom{0}
\end{tabular}\endgroup%
\kern3pt%
\begingroup \smaller\smaller\smaller\begin{tabular}{@{}c@{}}%
\phantom{0}\\\phantom{0}\\\phantom{0}\\10
\end{tabular}\endgroup%
{$\left.\llap{\phantom{%
\begingroup \smaller\smaller\smaller\begin{tabular}{@{}c@{}}%
\phantom{0}\\\phantom{0}\\\phantom{0}\\\phantom{0}
\end{tabular}\endgroup%
}}\!\right]$}%
{$\left[\!\llap{\phantom{%
\begingroup \smaller\smaller\smaller\begin{tabular}{@{}c@{}}%
0\\0\\0\\0
\end{tabular}\endgroup%
}}\right.$}%
\begingroup \smaller\smaller\smaller\begin{tabular}{@{}c@{}}%
4\\0\\1\\-1
\end{tabular}\endgroup%
\kern3pt%
\begingroup \smaller\smaller\smaller\begin{tabular}{@{}c@{}}%
15\\-2\\4\\-2
\end{tabular}\endgroup%
{$\left.\llap{\phantom{%
\begingroup \smaller\smaller\smaller\begin{tabular}{@{}c@{}}%
0\\0\\0\\0
\end{tabular}\endgroup%
}}\!\right]$}%
}%
\ifdim\wd\matricesbox>\halfwidth\myboxwidth=\hsize\else\myboxwidth=\halfwidth\fi
\vbox{%
\ifdim\myboxwidth=\hsize
\setbox\onelinebox=\hbox{%
\vbox{\hbox{%
$\Pi_{12,1}$ spans $L_{31.7}$%
}\hbox{%
$|2|2|2|2|2|2|2|2|2|2|2|2\rtimes D_{12}$%
}%
}%
\hfill\copy\matricesbox
}%
\ifdim\wd\onelinebox>\myboxwidth
\hbox to \myboxwidth{%
$\Pi_{12,1}$ spans $L_{31.7}$%
\hfil
$|2|2|2|2|2|2|2|2|2|2|2|2\rtimes D_{12}$%
}%
\box\matricesbox
\else
\hbox to \myboxwidth{%
\unhbox\onelinebox
}%
\fi
\else
\hbox to \myboxwidth{%
$\Pi_{12,1}$ spans $L_{31.7}$%
\hfil}%
\hbox to \myboxwidth{%
$|2|2|2|2|2|2|2|2|2|2|2|2\rtimes D_{12}$%
\hfil}%
\box\matricesbox
\fi
}%
\hfill\discretionary{}{}{}%
\setbox\matricesbox=\hbox{%
{$\left[\!\llap{\phantom{%
\begingroup \smaller\smaller\smaller\begin{tabular}{@{}c@{}}%
\phantom{0}\\\phantom{0}\\\phantom{0}\\\phantom{0}
\end{tabular}\endgroup%
}}\right.$}%
\begingroup \smaller\smaller\smaller\begin{tabular}{@{}c@{}}%
-1\\\phantom{0}\\\phantom{0}\\\phantom{0}
\end{tabular}\endgroup%
\kern3pt%
\begingroup \smaller\smaller\smaller\begin{tabular}{@{}c@{}}%
\phantom{0}\\15\\\phantom{0}\\\phantom{0}
\end{tabular}\endgroup%
\kern3pt%
\begingroup \smaller\smaller\smaller\begin{tabular}{@{}c@{}}%
\phantom{0}\\\phantom{0}\\15\\\phantom{0}
\end{tabular}\endgroup%
\kern3pt%
\begingroup \smaller\smaller\smaller\begin{tabular}{@{}c@{}}%
\phantom{0}\\\phantom{0}\\\phantom{0}\\15
\end{tabular}\endgroup%
{$\left.\llap{\phantom{%
\begingroup \smaller\smaller\smaller\begin{tabular}{@{}c@{}}%
\phantom{0}\\\phantom{0}\\\phantom{0}\\\phantom{0}
\end{tabular}\endgroup%
}}\!\right]$}%
{$\left[\!\llap{\phantom{%
\begingroup \smaller\smaller\smaller\begin{tabular}{@{}c@{}}%
0\\0\\0\\0
\end{tabular}\endgroup%
}}\right.$}%
\begingroup \smaller\smaller\smaller\begin{tabular}{@{}c@{}}%
5\\1\\0\\-1
\end{tabular}\endgroup%
\kern3pt%
\begingroup \smaller\smaller\smaller\begin{tabular}{@{}c@{}}%
9\\1\\1\\-2
\end{tabular}\endgroup%
{$\left.\llap{\phantom{%
\begingroup \smaller\smaller\smaller\begin{tabular}{@{}c@{}}%
0\\0\\0\\0
\end{tabular}\endgroup%
}}\!\right]$}%
}%
\ifdim\wd\matricesbox>\halfwidth\myboxwidth=\hsize\else\myboxwidth=\halfwidth\fi
\vbox{%
\ifdim\myboxwidth=\hsize
\setbox\onelinebox=\hbox{%
\vbox{\hbox{%
$\Pi_{12,2}$ spans $L_{16.13}$%
}\hbox{%
$|2|2|2|2|2|2|2|2|2|2|2|2\rtimes D_{12}$%
}%
}%
\hfill\copy\matricesbox
}%
\ifdim\wd\onelinebox>\myboxwidth
\hbox to \myboxwidth{%
$\Pi_{12,2}$ spans $L_{16.13}$%
\hfil
$|2|2|2|2|2|2|2|2|2|2|2|2\rtimes D_{12}$%
}%
\box\matricesbox
\else
\hbox to \myboxwidth{%
\unhbox\onelinebox
}%
\fi
\else
\hbox to \myboxwidth{%
$\Pi_{12,2}$ spans $L_{16.13}$%
\hfil}%
\hbox to \myboxwidth{%
$|2|2|2|2|2|2|2|2|2|2|2|2\rtimes D_{12}$%
\hfil}%
\box\matricesbox
\fi
}%
\hfill\discretionary{}{}{}%
\setbox\matricesbox=\hbox{%
{$\left[\!\llap{\phantom{%
\begingroup \smaller\smaller\smaller\begin{tabular}{@{}c@{}}%
\phantom{0}\\\phantom{0}\\\phantom{0}\\\phantom{0}
\end{tabular}\endgroup%
}}\right.$}%
\begingroup \smaller\smaller\smaller\begin{tabular}{@{}c@{}}%
-1\\\phantom{0}\\\phantom{0}\\\phantom{0}
\end{tabular}\endgroup%
\kern3pt%
\begingroup \smaller\smaller\smaller\begin{tabular}{@{}c@{}}%
\phantom{0}\\7\\\phantom{0}\\\phantom{0}
\end{tabular}\endgroup%
\kern3pt%
\begingroup \smaller\smaller\smaller\begin{tabular}{@{}c@{}}%
\phantom{0}\\\phantom{0}\\7\\\phantom{0}
\end{tabular}\endgroup%
\kern3pt%
\begingroup \smaller\smaller\smaller\begin{tabular}{@{}c@{}}%
\phantom{0}\\\phantom{0}\\\phantom{0}\\7
\end{tabular}\endgroup%
{$\left.\llap{\phantom{%
\begingroup \smaller\smaller\smaller\begin{tabular}{@{}c@{}}%
\phantom{0}\\\phantom{0}\\\phantom{0}\\\phantom{0}
\end{tabular}\endgroup%
}}\!\right]$}%
{$\left[\!\llap{\phantom{%
\begingroup \smaller\smaller\smaller\begin{tabular}{@{}c@{}}%
0\\0\\0\\0
\end{tabular}\endgroup%
}}\right.$}%
\begingroup \smaller\smaller\smaller\begin{tabular}{@{}c@{}}%
6\\1\\-2\\1
\end{tabular}\endgroup%
\kern3pt%
\begingroup \smaller\smaller\smaller\begin{tabular}{@{}c@{}}%
7\\2\\-2\\0
\end{tabular}\endgroup%
{$\left.\llap{\phantom{%
\begingroup \smaller\smaller\smaller\begin{tabular}{@{}c@{}}%
0\\0\\0\\0
\end{tabular}\endgroup%
}}\!\right]$}%
}%
\ifdim\wd\matricesbox>\halfwidth\myboxwidth=\hsize\else\myboxwidth=\halfwidth\fi
\vbox{%
\ifdim\myboxwidth=\hsize
\setbox\onelinebox=\hbox{%
\vbox{\hbox{%
$\Pi_{12,3}$ spans $L_{22.4}$%
}\hbox{%
$|2|2|2|2|2|2|2|2|2|2|2|2\rtimes D_{12}$%
}%
}%
\hfill\copy\matricesbox
}%
\ifdim\wd\onelinebox>\myboxwidth
\hbox to \myboxwidth{%
$\Pi_{12,3}$ spans $L_{22.4}$%
\hfil
$|2|2|2|2|2|2|2|2|2|2|2|2\rtimes D_{12}$%
}%
\box\matricesbox
\else
\hbox to \myboxwidth{%
\unhbox\onelinebox
}%
\fi
\else
\hbox to \myboxwidth{%
$\Pi_{12,3}$ spans $L_{22.4}$%
\hfil}%
\hbox to \myboxwidth{%
$|2|2|2|2|2|2|2|2|2|2|2|2\rtimes D_{12}$%
\hfil}%
\box\matricesbox
\fi
}%
\hfill\discretionary{}{}{}%
\setbox\matricesbox=\hbox{%
{$\left[\!\llap{\phantom{%
\begingroup \smaller\smaller\smaller\begin{tabular}{@{}c@{}}%
\phantom{0}\\\phantom{0}\\\phantom{0}\\\phantom{0}
\end{tabular}\endgroup%
}}\right.$}%
\begingroup \smaller\smaller\smaller\begin{tabular}{@{}c@{}}%
-6\\\phantom{0}\\\phantom{0}\\\phantom{0}
\end{tabular}\endgroup%
\kern3pt%
\begingroup \smaller\smaller\smaller\begin{tabular}{@{}c@{}}%
\phantom{0}\\1\\\phantom{0}\\\phantom{0}
\end{tabular}\endgroup%
\kern3pt%
\begingroup \smaller\smaller\smaller\begin{tabular}{@{}c@{}}%
\phantom{0}\\\phantom{0}\\1\\\phantom{0}
\end{tabular}\endgroup%
\kern3pt%
\begingroup \smaller\smaller\smaller\begin{tabular}{@{}c@{}}%
\phantom{0}\\\phantom{0}\\\phantom{0}\\1
\end{tabular}\endgroup%
{$\left.\llap{\phantom{%
\begingroup \smaller\smaller\smaller\begin{tabular}{@{}c@{}}%
\phantom{0}\\\phantom{0}\\\phantom{0}\\\phantom{0}
\end{tabular}\endgroup%
}}\!\right]$}%
{$\left[\!\llap{\phantom{%
\begingroup \smaller\smaller\smaller\begin{tabular}{@{}c@{}}%
0\\0\\0\\0
\end{tabular}\endgroup%
}}\right.$}%
\begingroup \smaller\smaller\smaller\begin{tabular}{@{}c@{}}%
2\\3\\1\\-4
\end{tabular}\endgroup%
{$\left.\llap{\phantom{%
\begingroup \smaller\smaller\smaller\begin{tabular}{@{}c@{}}%
0\\0\\0\\0
\end{tabular}\endgroup%
}}\!\right]$}%
}%
\ifdim\wd\matricesbox>\halfwidth\myboxwidth=\hsize\else\myboxwidth=\halfwidth\fi
\vbox{%
\ifdim\myboxwidth=\hsize
\setbox\onelinebox=\hbox{%
\vbox{\hbox{%
$\Pi_{12,4}=A_{3,III}=\hbox{GN}_{60}$ spans $L_{7.9}$%
}\hbox{%
$\slashthree\slashinfty\slashthree\slashinfty\slashthree\slashinfty\slashthree\slashinfty\slashthree\slashinfty\slashthree\slashinfty\rtimes D_{12}$%
}%
}%
\hfill\copy\matricesbox
}%
\ifdim\wd\onelinebox>\myboxwidth
\hbox to \myboxwidth{%
$\Pi_{12,4}=A_{3,III}=\hbox{GN}_{60}$ spans $L_{7.9}$%
\hfil
$\slashthree\slashinfty\slashthree\slashinfty\slashthree\slashinfty\slashthree\slashinfty\slashthree\slashinfty\slashthree\slashinfty\rtimes D_{12}$%
}%
\box\matricesbox
\else
\hbox to \myboxwidth{%
\unhbox\onelinebox
}%
\fi
\else
\hbox to \myboxwidth{%
$\Pi_{12,4}=A_{3,III}=\hbox{GN}_{60}$ spans $L_{7.9}$%
\hfil}%
\hbox to \myboxwidth{%
$\slashthree\slashinfty\slashthree\slashinfty\slashthree\slashinfty\slashthree\slashinfty\slashthree\slashinfty\slashthree\slashinfty\rtimes D_{12}$%
\hfil}%
\box\matricesbox
\fi
}%
\hfill\discretionary{}{}{}%
\setbox\matricesbox=\hbox{%
{$\left[\!\llap{\phantom{%
\begingroup \smaller\smaller\smaller\begin{tabular}{@{}c@{}}%
\phantom{0}\\\phantom{0}\\\phantom{0}\\\phantom{0}
\end{tabular}\endgroup%
}}\right.$}%
\begingroup \smaller\smaller\smaller\begin{tabular}{@{}c@{}}%
-4\\\phantom{0}\\\phantom{0}\\\phantom{0}
\end{tabular}\endgroup%
\kern3pt%
\begingroup \smaller\smaller\smaller\begin{tabular}{@{}c@{}}%
\phantom{0}\\3\\\phantom{0}\\\phantom{0}
\end{tabular}\endgroup%
\kern3pt%
\begingroup \smaller\smaller\smaller\begin{tabular}{@{}c@{}}%
\phantom{0}\\\phantom{0}\\3\\\phantom{0}
\end{tabular}\endgroup%
\kern3pt%
\begingroup \smaller\smaller\smaller\begin{tabular}{@{}c@{}}%
\phantom{0}\\\phantom{0}\\\phantom{0}\\3
\end{tabular}\endgroup%
{$\left.\llap{\phantom{%
\begingroup \smaller\smaller\smaller\begin{tabular}{@{}c@{}}%
\phantom{0}\\\phantom{0}\\\phantom{0}\\\phantom{0}
\end{tabular}\endgroup%
}}\!\right]$}%
{$\left[\!\llap{\phantom{%
\begingroup \smaller\smaller\smaller\begin{tabular}{@{}c@{}}%
0\\0\\0\\0
\end{tabular}\endgroup%
}}\right.$}%
\begingroup \smaller\smaller\smaller\begin{tabular}{@{}c@{}}%
2\\1\\-2\\1
\end{tabular}\endgroup%
\kern3pt%
\begingroup \smaller\smaller\smaller\begin{tabular}{@{}c@{}}%
6\\5\\-5\\0
\end{tabular}\endgroup%
{$\left.\llap{\phantom{%
\begingroup \smaller\smaller\smaller\begin{tabular}{@{}c@{}}%
0\\0\\0\\0
\end{tabular}\endgroup%
}}\!\right]$}%
}%
\ifdim\wd\matricesbox>\halfwidth\myboxwidth=\hsize\else\myboxwidth=\halfwidth\fi
\vbox{%
\ifdim\myboxwidth=\hsize
\setbox\onelinebox=\hbox{%
\vbox{\hbox{%
$\Pi_{12,5}$ spans $L_{168.1}$%
}\hbox{%
$|6|6|6|6|6|6|6|6|6|6|6|6\rtimes D_{12}$%
}%
}%
\hfill\copy\matricesbox
}%
\ifdim\wd\onelinebox>\myboxwidth
\hbox to \myboxwidth{%
$\Pi_{12,5}$ spans $L_{168.1}$%
\hfil
$|6|6|6|6|6|6|6|6|6|6|6|6\rtimes D_{12}$%
}%
\box\matricesbox
\else
\hbox to \myboxwidth{%
\unhbox\onelinebox
}%
\fi
\else
\hbox to \myboxwidth{%
$\Pi_{12,5}$ spans $L_{168.1}$%
\hfil}%
\hbox to \myboxwidth{%
$|6|6|6|6|6|6|6|6|6|6|6|6\rtimes D_{12}$%
\hfil}%
\box\matricesbox
\fi
}%
\hfill\discretionary{}{}{}%
\setbox\matricesbox=\hbox{%
{$\left[\!\llap{\phantom{%
\begingroup \smaller\smaller\smaller\begin{tabular}{@{}c@{}}%
\phantom{0}\\\phantom{0}\\\phantom{0}
\end{tabular}\endgroup%
}}\right.$}%
\begingroup \smaller\smaller\smaller\begin{tabular}{@{}c@{}}%
-5/4\\\phantom{0}\\\phantom{0}
\end{tabular}\endgroup%
\kern3pt%
\begingroup \smaller\smaller\smaller\begin{tabular}{@{}c@{}}%
\phantom{0}\\3\\\phantom{0}
\end{tabular}\endgroup%
\kern3pt%
\begingroup \smaller\smaller\smaller\begin{tabular}{@{}c@{}}%
\phantom{0}\\\phantom{0}\\3
\end{tabular}\endgroup%
{$\left.\llap{\phantom{%
\begingroup \smaller\smaller\smaller\begin{tabular}{@{}c@{}}%
\phantom{0}\\\phantom{0}\\\phantom{0}
\end{tabular}\endgroup%
}}\!\right]$}%
{$\left[\!\llap{\phantom{%
\begingroup \smaller\smaller\smaller\begin{tabular}{@{}c@{}}%
0\\0\\0
\end{tabular}\endgroup%
}}\right.$}%
\begingroup \smaller\smaller\smaller\begin{tabular}{@{}c@{}}%
6\\1\\-4
\end{tabular}\endgroup%
\kern3pt%
\begingroup \smaller\smaller\smaller\begin{tabular}{@{}c@{}}%
4\\2\\-2
\end{tabular}\endgroup%
{$\left.\llap{\phantom{%
\begingroup \smaller\smaller\smaller\begin{tabular}{@{}c@{}}%
0\\0\\0
\end{tabular}\endgroup%
}}\!\right]$}%
}%
\ifdim\wd\matricesbox>\halfwidth\myboxwidth=\hsize\else\myboxwidth=\halfwidth\fi
\vbox{%
\ifdim\myboxwidth=\hsize
\setbox\onelinebox=\hbox{%
\vbox{\hbox{%
$\Pi_{12,6}$ spans $L_{19.5}$%
}\hbox{%
$\slashtwo2|2\slashtwo2|2\slashtwo2|2\slashtwo2|2\rtimes D_{8}$%
}%
}%
\hfill\copy\matricesbox
}%
\ifdim\wd\onelinebox>\myboxwidth
\hbox to \myboxwidth{%
$\Pi_{12,6}$ spans $L_{19.5}$%
\hfil
$\slashtwo2|2\slashtwo2|2\slashtwo2|2\slashtwo2|2\rtimes D_{8}$%
}%
\box\matricesbox
\else
\hbox to \myboxwidth{%
\unhbox\onelinebox
}%
\fi
\else
\hbox to \myboxwidth{%
$\Pi_{12,6}$ spans $L_{19.5}$%
\hfil}%
\hbox to \myboxwidth{%
$\slashtwo2|2\slashtwo2|2\slashtwo2|2\slashtwo2|2\rtimes D_{8}$%
\hfil}%
\box\matricesbox
\fi
}%
\hfill\discretionary{}{}{}%
\setbox\matricesbox=\hbox{%
{$\left[\!\llap{\phantom{%
\begingroup \smaller\smaller\smaller\begin{tabular}{@{}c@{}}%
\phantom{0}\\\phantom{0}\\\phantom{0}
\end{tabular}\endgroup%
}}\right.$}%
\begingroup \smaller\smaller\smaller\begin{tabular}{@{}c@{}}%
-2\\\phantom{0}\\\phantom{0}
\end{tabular}\endgroup%
\kern3pt%
\begingroup \smaller\smaller\smaller\begin{tabular}{@{}c@{}}%
\phantom{0}\\3\\\phantom{0}
\end{tabular}\endgroup%
\kern3pt%
\begingroup \smaller\smaller\smaller\begin{tabular}{@{}c@{}}%
\phantom{0}\\\phantom{0}\\3
\end{tabular}\endgroup%
{$\left.\llap{\phantom{%
\begingroup \smaller\smaller\smaller\begin{tabular}{@{}c@{}}%
\phantom{0}\\\phantom{0}\\\phantom{0}
\end{tabular}\endgroup%
}}\!\right]$}%
{$\left[\!\llap{\phantom{%
\begingroup \smaller\smaller\smaller\begin{tabular}{@{}c@{}}%
0\\0\\0
\end{tabular}\endgroup%
}}\right.$}%
\begingroup \smaller\smaller\smaller\begin{tabular}{@{}c@{}}%
1\\0\\1
\end{tabular}\endgroup%
\kern3pt%
\begingroup \smaller\smaller\smaller\begin{tabular}{@{}c@{}}%
12\\-6\\8
\end{tabular}\endgroup%
{$\left.\llap{\phantom{%
\begingroup \smaller\smaller\smaller\begin{tabular}{@{}c@{}}%
0\\0\\0
\end{tabular}\endgroup%
}}\!\right]$}%
}%
\ifdim\wd\matricesbox>\halfwidth\myboxwidth=\hsize\else\myboxwidth=\halfwidth\fi
\vbox{%
\ifdim\myboxwidth=\hsize
\setbox\onelinebox=\hbox{%
\vbox{\hbox{%
$\Pi_{12,7}$ spans $L_{123.11}$%
}\hbox{%
$|2\slashtwo2|2\slashtwo2|2\slashtwo2|2\slashtwo2\rtimes D_{8}$%
}%
}%
\hfill\copy\matricesbox
}%
\ifdim\wd\onelinebox>\myboxwidth
\hbox to \myboxwidth{%
$\Pi_{12,7}$ spans $L_{123.11}$%
\hfil
$|2\slashtwo2|2\slashtwo2|2\slashtwo2|2\slashtwo2\rtimes D_{8}$%
}%
\box\matricesbox
\else
\hbox to \myboxwidth{%
\unhbox\onelinebox
}%
\fi
\else
\hbox to \myboxwidth{%
$\Pi_{12,7}$ spans $L_{123.11}$%
\hfil}%
\hbox to \myboxwidth{%
$|2\slashtwo2|2\slashtwo2|2\slashtwo2|2\slashtwo2\rtimes D_{8}$%
\hfil}%
\box\matricesbox
\fi
}%
\hfill\discretionary{}{}{}%
\setbox\matricesbox=\hbox{%
{$\left[\!\llap{\phantom{%
\begingroup \smaller\smaller\smaller\begin{tabular}{@{}c@{}}%
\phantom{0}\\\phantom{0}\\\phantom{0}
\end{tabular}\endgroup%
}}\right.$}%
\begingroup \smaller\smaller\smaller\begin{tabular}{@{}c@{}}%
-3\\\phantom{0}\\\phantom{0}
\end{tabular}\endgroup%
\kern3pt%
\begingroup \smaller\smaller\smaller\begin{tabular}{@{}c@{}}%
\phantom{0}\\4\\\phantom{0}
\end{tabular}\endgroup%
\kern3pt%
\begingroup \smaller\smaller\smaller\begin{tabular}{@{}c@{}}%
\phantom{0}\\\phantom{0}\\4
\end{tabular}\endgroup%
{$\left.\llap{\phantom{%
\begingroup \smaller\smaller\smaller\begin{tabular}{@{}c@{}}%
\phantom{0}\\\phantom{0}\\\phantom{0}
\end{tabular}\endgroup%
}}\!\right]$}%
{$\left[\!\llap{\phantom{%
\begingroup \smaller\smaller\smaller\begin{tabular}{@{}c@{}}%
0\\0\\0
\end{tabular}\endgroup%
}}\right.$}%
\begingroup \smaller\smaller\smaller\begin{tabular}{@{}c@{}}%
1\\-1\\0
\end{tabular}\endgroup%
\kern3pt%
\begingroup \smaller\smaller\smaller\begin{tabular}{@{}c@{}}%
4\\-3\\-2
\end{tabular}\endgroup%
{$\left.\llap{\phantom{%
\begingroup \smaller\smaller\smaller\begin{tabular}{@{}c@{}}%
0\\0\\0
\end{tabular}\endgroup%
}}\!\right]$}%
}%
\ifdim\wd\matricesbox>\halfwidth\myboxwidth=\hsize\else\myboxwidth=\halfwidth\fi
\vbox{%
\ifdim\myboxwidth=\hsize
\setbox\onelinebox=\hbox{%
\vbox{\hbox{%
$\Pi_{12,8}=\hbox{GN}_{58}$ spans $L_{123.9}$%
}\hbox{%
$|2\slashtwo2|2\slashtwo2|2\slashtwo2|2\slashtwo2\rtimes D_{8}$%
}%
}%
\hfill\copy\matricesbox
}%
\ifdim\wd\onelinebox>\myboxwidth
\hbox to \myboxwidth{%
$\Pi_{12,8}=\hbox{GN}_{58}$ spans $L_{123.9}$%
\hfil
$|2\slashtwo2|2\slashtwo2|2\slashtwo2|2\slashtwo2\rtimes D_{8}$%
}%
\box\matricesbox
\else
\hbox to \myboxwidth{%
\unhbox\onelinebox
}%
\fi
\else
\hbox to \myboxwidth{%
$\Pi_{12,8}=\hbox{GN}_{58}$ spans $L_{123.9}$%
\hfil}%
\hbox to \myboxwidth{%
$|2\slashtwo2|2\slashtwo2|2\slashtwo2|2\slashtwo2\rtimes D_{8}$%
\hfil}%
\box\matricesbox
\fi
}%
\hfill\discretionary{}{}{}%
\setbox\matricesbox=\hbox{%
{$\left[\!\llap{\phantom{%
\begingroup \smaller\smaller\smaller\begin{tabular}{@{}c@{}}%
\phantom{0}\\\phantom{0}\\\phantom{0}
\end{tabular}\endgroup%
}}\right.$}%
\begingroup \smaller\smaller\smaller\begin{tabular}{@{}c@{}}%
-8\\\phantom{0}\\\phantom{0}
\end{tabular}\endgroup%
\kern3pt%
\begingroup \smaller\smaller\smaller\begin{tabular}{@{}c@{}}%
\phantom{0}\\1\\\phantom{0}
\end{tabular}\endgroup%
\kern3pt%
\begingroup \smaller\smaller\smaller\begin{tabular}{@{}c@{}}%
\phantom{0}\\\phantom{0}\\1
\end{tabular}\endgroup%
{$\left.\llap{\phantom{%
\begingroup \smaller\smaller\smaller\begin{tabular}{@{}c@{}}%
\phantom{0}\\\phantom{0}\\\phantom{0}
\end{tabular}\endgroup%
}}\!\right]$}%
{$\left[\!\llap{\phantom{%
\begingroup \smaller\smaller\smaller\begin{tabular}{@{}c@{}}%
0\\0\\0
\end{tabular}\endgroup%
}}\right.$}%
\begingroup \smaller\smaller\smaller\begin{tabular}{@{}c@{}}%
1\\0\\-3
\end{tabular}\endgroup%
\kern3pt%
\begingroup \smaller\smaller\smaller\begin{tabular}{@{}c@{}}%
2\\3\\-5
\end{tabular}\endgroup%
{$\left.\llap{\phantom{%
\begingroup \smaller\smaller\smaller\begin{tabular}{@{}c@{}}%
0\\0\\0
\end{tabular}\endgroup%
}}\!\right]$}%
}%
\ifdim\wd\matricesbox>\halfwidth\myboxwidth=\hsize\else\myboxwidth=\halfwidth\fi
\vbox{%
\ifdim\myboxwidth=\hsize
\setbox\onelinebox=\hbox{%
\vbox{\hbox{%
$\Pi_{12,9}$ spans $L_{159.1}$%
}\hbox{%
$|4\slashinfty4|4\slashinfty4|4\slashinfty4|4\slashinfty4\rtimes D_{8}$%
}%
}%
\hfill\copy\matricesbox
}%
\ifdim\wd\onelinebox>\myboxwidth
\hbox to \myboxwidth{%
$\Pi_{12,9}$ spans $L_{159.1}$%
\hfil
$|4\slashinfty4|4\slashinfty4|4\slashinfty4|4\slashinfty4\rtimes D_{8}$%
}%
\box\matricesbox
\else
\hbox to \myboxwidth{%
\unhbox\onelinebox
}%
\fi
\else
\hbox to \myboxwidth{%
$\Pi_{12,9}$ spans $L_{159.1}$%
\hfil}%
\hbox to \myboxwidth{%
$|4\slashinfty4|4\slashinfty4|4\slashinfty4|4\slashinfty4\rtimes D_{8}$%
\hfil}%
\box\matricesbox
\fi
}%
\hfill\discretionary{}{}{}%
\setbox\matricesbox=\hbox{%
{$\left[\!\llap{\phantom{%
\begingroup \smaller\smaller\smaller\begin{tabular}{@{}c@{}}%
\phantom{0}\\\phantom{0}\\\phantom{0}
\end{tabular}\endgroup%
}}\right.$}%
\begingroup \smaller\smaller\smaller\begin{tabular}{@{}c@{}}%
-1\\\phantom{0}\\\phantom{0}
\end{tabular}\endgroup%
\kern3pt%
\begingroup \smaller\smaller\smaller\begin{tabular}{@{}c@{}}%
\phantom{0}\\18\\\phantom{0}
\end{tabular}\endgroup%
\kern3pt%
\begingroup \smaller\smaller\smaller\begin{tabular}{@{}c@{}}%
\phantom{0}\\\phantom{0}\\18
\end{tabular}\endgroup%
{$\left.\llap{\phantom{%
\begingroup \smaller\smaller\smaller\begin{tabular}{@{}c@{}}%
\phantom{0}\\\phantom{0}\\\phantom{0}
\end{tabular}\endgroup%
}}\!\right]$}%
{$\left[\!\llap{\phantom{%
\begingroup \smaller\smaller\smaller\begin{tabular}{@{}c@{}}%
0\\0\\0
\end{tabular}\endgroup%
}}\right.$}%
\begingroup \smaller\smaller\smaller\begin{tabular}{@{}c@{}}%
9\\1\\2
\end{tabular}\endgroup%
\kern3pt%
\begingroup \smaller\smaller\smaller\begin{tabular}{@{}c@{}}%
8\\0\\2
\end{tabular}\endgroup%
{$\left.\llap{\phantom{%
\begingroup \smaller\smaller\smaller\begin{tabular}{@{}c@{}}%
0\\0\\0
\end{tabular}\endgroup%
}}\!\right]$}%
}%
\ifdim\wd\matricesbox>\halfwidth\myboxwidth=\hsize\else\myboxwidth=\halfwidth\fi
\vbox{%
\ifdim\myboxwidth=\hsize
\setbox\onelinebox=\hbox{%
\vbox{\hbox{%
$\Pi_{12,10}$ spans $L_{142.20}$%
}\hbox{%
$\slashinfty2|2\slashinfty2|2\slashinfty2|2\slashinfty2|2\rtimes D_{8}$%
}%
}%
\hfill\copy\matricesbox
}%
\ifdim\wd\onelinebox>\myboxwidth
\hbox to \myboxwidth{%
$\Pi_{12,10}$ spans $L_{142.20}$%
\hfil
$\slashinfty2|2\slashinfty2|2\slashinfty2|2\slashinfty2|2\rtimes D_{8}$%
}%
\box\matricesbox
\else
\hbox to \myboxwidth{%
\unhbox\onelinebox
}%
\fi
\else
\hbox to \myboxwidth{%
$\Pi_{12,10}$ spans $L_{142.20}$%
\hfil}%
\hbox to \myboxwidth{%
$\slashinfty2|2\slashinfty2|2\slashinfty2|2\slashinfty2|2\rtimes D_{8}$%
\hfil}%
\box\matricesbox
\fi
}%
\hfill\discretionary{}{}{}%
\setbox\matricesbox=\hbox{%
{$\left[\!\llap{\phantom{%
\begingroup \smaller\smaller\smaller\begin{tabular}{@{}c@{}}%
\phantom{0}\\\phantom{0}\\\phantom{0}\\\phantom{0}
\end{tabular}\endgroup%
}}\right.$}%
\begingroup \smaller\smaller\smaller\begin{tabular}{@{}c@{}}%
-3\\\phantom{0}\\\phantom{0}\\\phantom{0}
\end{tabular}\endgroup%
\kern3pt%
\begingroup \smaller\smaller\smaller\begin{tabular}{@{}c@{}}%
\phantom{0}\\1\\\phantom{0}\\\phantom{0}
\end{tabular}\endgroup%
\kern3pt%
\begingroup \smaller\smaller\smaller\begin{tabular}{@{}c@{}}%
\phantom{0}\\\phantom{0}\\1\\\phantom{0}
\end{tabular}\endgroup%
\kern3pt%
\begingroup \smaller\smaller\smaller\begin{tabular}{@{}c@{}}%
\phantom{0}\\\phantom{0}\\\phantom{0}\\1
\end{tabular}\endgroup%
{$\left.\llap{\phantom{%
\begingroup \smaller\smaller\smaller\begin{tabular}{@{}c@{}}%
\phantom{0}\\\phantom{0}\\\phantom{0}\\\phantom{0}
\end{tabular}\endgroup%
}}\!\right]$}%
{$\left[\!\llap{\phantom{%
\begingroup \smaller\smaller\smaller\begin{tabular}{@{}c@{}}%
0\\0\\0\\0
\end{tabular}\endgroup%
}}\right.$}%
\begingroup \smaller\smaller\smaller\begin{tabular}{@{}c@{}}%
2\\-3\\2\\1
\end{tabular}\endgroup%
\kern3pt%
\begingroup \smaller\smaller\smaller\begin{tabular}{@{}c@{}}%
6\\-8\\1\\7
\end{tabular}\endgroup%
{$\left.\llap{\phantom{%
\begingroup \smaller\smaller\smaller\begin{tabular}{@{}c@{}}%
0\\0\\0\\0
\end{tabular}\endgroup%
}}\!\right]$}%
}%
\ifdim\wd\matricesbox>\halfwidth\myboxwidth=\hsize\else\myboxwidth=\halfwidth\fi
\vbox{%
\ifdim\myboxwidth=\hsize
\setbox\onelinebox=\hbox{%
\vbox{\hbox{%
$\Pi_{12,11}$ spans $L_{3.4}$%
}\hbox{%
$626262626262\rtimes C_{6}$%
}%
}%
\hfill\copy\matricesbox
}%
\ifdim\wd\onelinebox>\myboxwidth
\hbox to \myboxwidth{%
$\Pi_{12,11}$ spans $L_{3.4}$%
\hfil
$626262626262\rtimes C_{6}$%
}%
\box\matricesbox
\else
\hbox to \myboxwidth{%
\unhbox\onelinebox
}%
\fi
\else
\hbox to \myboxwidth{%
$\Pi_{12,11}$ spans $L_{3.4}$%
\hfil}%
\hbox to \myboxwidth{%
$626262626262\rtimes C_{6}$%
\hfil}%
\box\matricesbox
\fi
}%
\hfill\discretionary{}{}{}%
\setbox\matricesbox=\hbox{%
{$\left[\!\llap{\phantom{%
\begingroup \smaller\smaller\smaller\begin{tabular}{@{}c@{}}%
\phantom{0}\\\phantom{0}\\\phantom{0}
\end{tabular}\endgroup%
}}\right.$}%
\begingroup \smaller\smaller\smaller\begin{tabular}{@{}c@{}}%
-1\\\phantom{0}\\\phantom{0}
\end{tabular}\endgroup%
\kern3pt%
\begingroup \smaller\smaller\smaller\begin{tabular}{@{}c@{}}%
\phantom{0}\\1/2\\\phantom{0}
\end{tabular}\endgroup%
\kern3pt%
\begingroup \smaller\smaller\smaller\begin{tabular}{@{}c@{}}%
\phantom{0}\\\phantom{0}\\195/2
\end{tabular}\endgroup%
{$\left.\llap{\phantom{%
\begingroup \smaller\smaller\smaller\begin{tabular}{@{}c@{}}%
\phantom{0}\\\phantom{0}\\\phantom{0}
\end{tabular}\endgroup%
}}\!\right]$}%
{$\left[\!\llap{\phantom{%
\begingroup \smaller\smaller\smaller\begin{tabular}{@{}c@{}}%
0\\0\\0
\end{tabular}\endgroup%
}}\right.$}%
\begingroup \smaller\smaller\smaller\begin{tabular}{@{}c@{}}%
1\\-2\\0
\end{tabular}\endgroup%
\kern3pt%
\begingroup \smaller\smaller\smaller\begin{tabular}{@{}c@{}}%
13\\-13\\-1
\end{tabular}\endgroup%
\kern3pt%
\begingroup \smaller\smaller\smaller\begin{tabular}{@{}c@{}}%
10\\-5\\-1
\end{tabular}\endgroup%
\kern3pt%
\begingroup \smaller\smaller\smaller\begin{tabular}{@{}c@{}}%
39\\0\\-4
\end{tabular}\endgroup%
{$\left.\llap{\phantom{%
\begingroup \smaller\smaller\smaller\begin{tabular}{@{}c@{}}%
0\\0\\0
\end{tabular}\endgroup%
}}\!\right]$}%
}%
\ifdim\wd\matricesbox>\halfwidth\myboxwidth=\hsize\else\myboxwidth=\halfwidth\fi
\vbox{%
\ifdim\myboxwidth=\hsize
\setbox\onelinebox=\hbox{%
\vbox{\hbox{%
$\Pi_{12,12}$ spans $L_{84.2}$%
}\hbox{%
$|222|222|222|222\rtimes D_{4}$%
}%
}%
\hfill\copy\matricesbox
}%
\ifdim\wd\onelinebox>\myboxwidth
\hbox to \myboxwidth{%
$\Pi_{12,12}$ spans $L_{84.2}$%
\hfil
$|222|222|222|222\rtimes D_{4}$%
}%
\box\matricesbox
\else
\hbox to \myboxwidth{%
\unhbox\onelinebox
}%
\fi
\else
\hbox to \myboxwidth{%
$\Pi_{12,12}$ spans $L_{84.2}$%
\hfil}%
\hbox to \myboxwidth{%
$|222|222|222|222\rtimes D_{4}$%
\hfil}%
\box\matricesbox
\fi
}%
\hfill\discretionary{}{}{}%
\setbox\matricesbox=\hbox{%
{$\left[\!\llap{\phantom{%
\begingroup \smaller\smaller\smaller\begin{tabular}{@{}c@{}}%
\phantom{0}\\\phantom{0}\\\phantom{0}
\end{tabular}\endgroup%
}}\right.$}%
\begingroup \smaller\smaller\smaller\begin{tabular}{@{}c@{}}%
-1/2\\\phantom{0}\\\phantom{0}
\end{tabular}\endgroup%
\kern3pt%
\begingroup \smaller\smaller\smaller\begin{tabular}{@{}c@{}}%
\phantom{0}\\3/2\\\phantom{0}
\end{tabular}\endgroup%
\kern3pt%
\begingroup \smaller\smaller\smaller\begin{tabular}{@{}c@{}}%
\phantom{0}\\\phantom{0}\\36
\end{tabular}\endgroup%
{$\left.\llap{\phantom{%
\begingroup \smaller\smaller\smaller\begin{tabular}{@{}c@{}}%
\phantom{0}\\\phantom{0}\\\phantom{0}
\end{tabular}\endgroup%
}}\!\right]$}%
{$\left[\!\llap{\phantom{%
\begingroup \smaller\smaller\smaller\begin{tabular}{@{}c@{}}%
0\\0\\0
\end{tabular}\endgroup%
}}\right.$}%
\begingroup \smaller\smaller\smaller\begin{tabular}{@{}c@{}}%
6\\4\\0
\end{tabular}\endgroup%
\kern3pt%
\begingroup \smaller\smaller\smaller\begin{tabular}{@{}c@{}}%
64\\32\\4
\end{tabular}\endgroup%
\kern3pt%
\begingroup \smaller\smaller\smaller\begin{tabular}{@{}c@{}}%
9\\3\\1
\end{tabular}\endgroup%
\kern3pt%
\begingroup \smaller\smaller\smaller\begin{tabular}{@{}c@{}}%
16\\0\\2
\end{tabular}\endgroup%
{$\left.\llap{\phantom{%
\begingroup \smaller\smaller\smaller\begin{tabular}{@{}c@{}}%
0\\0\\0
\end{tabular}\endgroup%
}}\!\right]$}%
}%
\ifdim\wd\matricesbox>\halfwidth\myboxwidth=\hsize\else\myboxwidth=\halfwidth\fi
\vbox{%
\ifdim\myboxwidth=\hsize
\setbox\onelinebox=\hbox{%
\vbox{\hbox{%
$\Pi_{12,13}$ spans $L_{172.10}$%
}\hbox{%
$|222|222|222|222\rtimes D_{4}$%
}%
}%
\hfill\copy\matricesbox
}%
\ifdim\wd\onelinebox>\myboxwidth
\hbox to \myboxwidth{%
$\Pi_{12,13}$ spans $L_{172.10}$%
\hfil
$|222|222|222|222\rtimes D_{4}$%
}%
\box\matricesbox
\else
\hbox to \myboxwidth{%
\unhbox\onelinebox
}%
\fi
\else
\hbox to \myboxwidth{%
$\Pi_{12,13}$ spans $L_{172.10}$%
\hfil}%
\hbox to \myboxwidth{%
$|222|222|222|222\rtimes D_{4}$%
\hfil}%
\box\matricesbox
\fi
}%
\hfill\discretionary{}{}{}%
\setbox\matricesbox=\hbox{%
{$\left[\!\llap{\phantom{%
\begingroup \smaller\smaller\smaller\begin{tabular}{@{}c@{}}%
\phantom{0}\\\phantom{0}\\\phantom{0}
\end{tabular}\endgroup%
}}\right.$}%
\begingroup \smaller\smaller\smaller\begin{tabular}{@{}c@{}}%
-2\\\phantom{0}\\\phantom{0}
\end{tabular}\endgroup%
\kern3pt%
\begingroup \smaller\smaller\smaller\begin{tabular}{@{}c@{}}%
\phantom{0}\\9\\\phantom{0}
\end{tabular}\endgroup%
\kern3pt%
\begingroup \smaller\smaller\smaller\begin{tabular}{@{}c@{}}%
\phantom{0}\\\phantom{0}\\3
\end{tabular}\endgroup%
{$\left.\llap{\phantom{%
\begingroup \smaller\smaller\smaller\begin{tabular}{@{}c@{}}%
\phantom{0}\\\phantom{0}\\\phantom{0}
\end{tabular}\endgroup%
}}\!\right]$}%
{$\left[\!\llap{\phantom{%
\begingroup \smaller\smaller\smaller\begin{tabular}{@{}c@{}}%
0\\0\\0
\end{tabular}\endgroup%
}}\right.$}%
\begingroup \smaller\smaller\smaller\begin{tabular}{@{}c@{}}%
4\\-2\\0
\end{tabular}\endgroup%
\kern3pt%
\begingroup \smaller\smaller\smaller\begin{tabular}{@{}c@{}}%
9\\-4\\-3
\end{tabular}\endgroup%
\kern3pt%
\begingroup \smaller\smaller\smaller\begin{tabular}{@{}c@{}}%
3\\-1\\-2
\end{tabular}\endgroup%
\kern3pt%
\begingroup \smaller\smaller\smaller\begin{tabular}{@{}c@{}}%
1\\0\\-1
\end{tabular}\endgroup%
{$\left.\llap{\phantom{%
\begingroup \smaller\smaller\smaller\begin{tabular}{@{}c@{}}%
0\\0\\0
\end{tabular}\endgroup%
}}\!\right]$}%
}%
\ifdim\wd\matricesbox>\halfwidth\myboxwidth=\hsize\else\myboxwidth=\halfwidth\fi
\vbox{%
\ifdim\myboxwidth=\hsize
\setbox\onelinebox=\hbox{%
\vbox{\hbox{%
$\Pi_{12,14}$ spans $L_{175.1}$%
}\hbox{%
$22|222|222|222|2\rtimes D_{4}$%
}%
}%
\hfill\copy\matricesbox
}%
\ifdim\wd\onelinebox>\myboxwidth
\hbox to \myboxwidth{%
$\Pi_{12,14}$ spans $L_{175.1}$%
\hfil
$22|222|222|222|2\rtimes D_{4}$%
}%
\box\matricesbox
\else
\hbox to \myboxwidth{%
\unhbox\onelinebox
}%
\fi
\else
\hbox to \myboxwidth{%
$\Pi_{12,14}$ spans $L_{175.1}$%
\hfil}%
\hbox to \myboxwidth{%
$22|222|222|222|2\rtimes D_{4}$%
\hfil}%
\box\matricesbox
\fi
}%
\hfill\discretionary{}{}{}%
\setbox\matricesbox=\hbox{%
{$\left[\!\llap{\phantom{%
\begingroup \smaller\smaller\smaller\begin{tabular}{@{}c@{}}%
\phantom{0}\\\phantom{0}\\\phantom{0}
\end{tabular}\endgroup%
}}\right.$}%
\begingroup \smaller\smaller\smaller\begin{tabular}{@{}c@{}}%
-1\\\phantom{0}\\\phantom{0}
\end{tabular}\endgroup%
\kern3pt%
\begingroup \smaller\smaller\smaller\begin{tabular}{@{}c@{}}%
\phantom{0}\\3\\\phantom{0}
\end{tabular}\endgroup%
\kern3pt%
\begingroup \smaller\smaller\smaller\begin{tabular}{@{}c@{}}%
\phantom{0}\\\phantom{0}\\15
\end{tabular}\endgroup%
{$\left.\llap{\phantom{%
\begingroup \smaller\smaller\smaller\begin{tabular}{@{}c@{}}%
\phantom{0}\\\phantom{0}\\\phantom{0}
\end{tabular}\endgroup%
}}\!\right]$}%
{$\left[\!\llap{\phantom{%
\begingroup \smaller\smaller\smaller\begin{tabular}{@{}c@{}}%
0\\0\\0
\end{tabular}\endgroup%
}}\right.$}%
\begingroup \smaller\smaller\smaller\begin{tabular}{@{}c@{}}%
3\\-2\\0
\end{tabular}\endgroup%
\kern3pt%
\begingroup \smaller\smaller\smaller\begin{tabular}{@{}c@{}}%
6\\-3\\1
\end{tabular}\endgroup%
\kern3pt%
\begingroup \smaller\smaller\smaller\begin{tabular}{@{}c@{}}%
8\\-2\\2
\end{tabular}\endgroup%
\kern3pt%
\begingroup \smaller\smaller\smaller\begin{tabular}{@{}c@{}}%
15\\0\\4
\end{tabular}\endgroup%
{$\left.\llap{\phantom{%
\begingroup \smaller\smaller\smaller\begin{tabular}{@{}c@{}}%
0\\0\\0
\end{tabular}\endgroup%
}}\!\right]$}%
}%
\ifdim\wd\matricesbox>\halfwidth\myboxwidth=\hsize\else\myboxwidth=\halfwidth\fi
\vbox{%
\ifdim\myboxwidth=\hsize
\setbox\onelinebox=\hbox{%
\vbox{\hbox{%
$\Pi_{12,15}$ spans $L_{128.12}$%
}\hbox{%
$|222|222|222|222\rtimes D_{4}$%
}%
}%
\hfill\copy\matricesbox
}%
\ifdim\wd\onelinebox>\myboxwidth
\hbox to \myboxwidth{%
$\Pi_{12,15}$ spans $L_{128.12}$%
\hfil
$|222|222|222|222\rtimes D_{4}$%
}%
\box\matricesbox
\else
\hbox to \myboxwidth{%
\unhbox\onelinebox
}%
\fi
\else
\hbox to \myboxwidth{%
$\Pi_{12,15}$ spans $L_{128.12}$%
\hfil}%
\hbox to \myboxwidth{%
$|222|222|222|222\rtimes D_{4}$%
\hfil}%
\box\matricesbox
\fi
}%
\hfill\discretionary{}{}{}%
\setbox\matricesbox=\hbox{%
{$\left[\!\llap{\phantom{%
\begingroup \smaller\smaller\smaller\begin{tabular}{@{}c@{}}%
\phantom{0}\\\phantom{0}\\\phantom{0}
\end{tabular}\endgroup%
}}\right.$}%
\begingroup \smaller\smaller\smaller\begin{tabular}{@{}c@{}}%
-1\\\phantom{0}\\\phantom{0}
\end{tabular}\endgroup%
\kern3pt%
\begingroup \smaller\smaller\smaller\begin{tabular}{@{}c@{}}%
\phantom{0}\\60\\\phantom{0}
\end{tabular}\endgroup%
\kern3pt%
\begingroup \smaller\smaller\smaller\begin{tabular}{@{}c@{}}%
\phantom{0}\\\phantom{0}\\2
\end{tabular}\endgroup%
{$\left.\llap{\phantom{%
\begingroup \smaller\smaller\smaller\begin{tabular}{@{}c@{}}%
\phantom{0}\\\phantom{0}\\\phantom{0}
\end{tabular}\endgroup%
}}\!\right]$}%
{$\left[\!\llap{\phantom{%
\begingroup \smaller\smaller\smaller\begin{tabular}{@{}c@{}}%
0\\0\\0
\end{tabular}\endgroup%
}}\right.$}%
\begingroup \smaller\smaller\smaller\begin{tabular}{@{}c@{}}%
15\\-2\\0
\end{tabular}\endgroup%
\kern3pt%
\begingroup \smaller\smaller\smaller\begin{tabular}{@{}c@{}}%
16\\-2\\-4
\end{tabular}\endgroup%
\kern3pt%
\begingroup \smaller\smaller\smaller\begin{tabular}{@{}c@{}}%
10\\-1\\-5
\end{tabular}\endgroup%
\kern3pt%
\begingroup \smaller\smaller\smaller\begin{tabular}{@{}c@{}}%
1\\0\\-1
\end{tabular}\endgroup%
{$\left.\llap{\phantom{%
\begingroup \smaller\smaller\smaller\begin{tabular}{@{}c@{}}%
0\\0\\0
\end{tabular}\endgroup%
}}\!\right]$}%
}%
\ifdim\wd\matricesbox>\halfwidth\myboxwidth=\hsize\else\myboxwidth=\halfwidth\fi
\vbox{%
\ifdim\myboxwidth=\hsize
\setbox\onelinebox=\hbox{%
\vbox{\hbox{%
$\Pi_{12,16}$ spans $L_{205.6}$%
}\hbox{%
$22|222|222|222|2\rtimes D_{4}$%
}%
}%
\hfill\copy\matricesbox
}%
\ifdim\wd\onelinebox>\myboxwidth
\hbox to \myboxwidth{%
$\Pi_{12,16}$ spans $L_{205.6}$%
\hfil
$22|222|222|222|2\rtimes D_{4}$%
}%
\box\matricesbox
\else
\hbox to \myboxwidth{%
\unhbox\onelinebox
}%
\fi
\else
\hbox to \myboxwidth{%
$\Pi_{12,16}$ spans $L_{205.6}$%
\hfil}%
\hbox to \myboxwidth{%
$22|222|222|222|2\rtimes D_{4}$%
\hfil}%
\box\matricesbox
\fi
}%
\hfill\discretionary{}{}{}%
\setbox\matricesbox=\hbox{%
{$\left[\!\llap{\phantom{%
\begingroup \smaller\smaller\smaller\begin{tabular}{@{}c@{}}%
\phantom{0}\\\phantom{0}\\\phantom{0}
\end{tabular}\endgroup%
}}\right.$}%
\begingroup \smaller\smaller\smaller\begin{tabular}{@{}c@{}}%
-1/4\\\phantom{0}\\\phantom{0}
\end{tabular}\endgroup%
\kern3pt%
\begingroup \smaller\smaller\smaller\begin{tabular}{@{}c@{}}%
\phantom{0}\\15\\\phantom{0}
\end{tabular}\endgroup%
\kern3pt%
\begingroup \smaller\smaller\smaller\begin{tabular}{@{}c@{}}%
\phantom{0}\\\phantom{0}\\195
\end{tabular}\endgroup%
{$\left.\llap{\phantom{%
\begingroup \smaller\smaller\smaller\begin{tabular}{@{}c@{}}%
\phantom{0}\\\phantom{0}\\\phantom{0}
\end{tabular}\endgroup%
}}\!\right]$}%
{$\left[\!\llap{\phantom{%
\begingroup \smaller\smaller\smaller\begin{tabular}{@{}c@{}}%
0\\0\\0
\end{tabular}\endgroup%
}}\right.$}%
\begingroup \smaller\smaller\smaller\begin{tabular}{@{}c@{}}%
6\\-1\\0
\end{tabular}\endgroup%
\kern3pt%
\begingroup \smaller\smaller\smaller\begin{tabular}{@{}c@{}}%
260\\-26\\-6
\end{tabular}\endgroup%
\kern3pt%
\begingroup \smaller\smaller\smaller\begin{tabular}{@{}c@{}}%
30\\-2\\-1
\end{tabular}\endgroup%
\kern3pt%
\begingroup \smaller\smaller\smaller\begin{tabular}{@{}c@{}}%
26\\0\\-1
\end{tabular}\endgroup%
{$\left.\llap{\phantom{%
\begingroup \smaller\smaller\smaller\begin{tabular}{@{}c@{}}%
0\\0\\0
\end{tabular}\endgroup%
}}\!\right]$}%
}%
\ifdim\wd\matricesbox>\halfwidth\myboxwidth=\hsize\else\myboxwidth=\halfwidth\fi
\vbox{%
\ifdim\myboxwidth=\hsize
\setbox\onelinebox=\hbox{%
\vbox{\hbox{%
$\Pi_{12,17}$ spans $L_{84.13}$%
}\hbox{%
$|222|222|222|222\rtimes D_{4}$%
}%
}%
\hfill\copy\matricesbox
}%
\ifdim\wd\onelinebox>\myboxwidth
\hbox to \myboxwidth{%
$\Pi_{12,17}$ spans $L_{84.13}$%
\hfil
$|222|222|222|222\rtimes D_{4}$%
}%
\box\matricesbox
\else
\hbox to \myboxwidth{%
\unhbox\onelinebox
}%
\fi
\else
\hbox to \myboxwidth{%
$\Pi_{12,17}$ spans $L_{84.13}$%
\hfil}%
\hbox to \myboxwidth{%
$|222|222|222|222\rtimes D_{4}$%
\hfil}%
\box\matricesbox
\fi
}%
\hfill\discretionary{}{}{}%
\setbox\matricesbox=\hbox{%
{$\left[\!\llap{\phantom{%
\begingroup \smaller\smaller\smaller\begin{tabular}{@{}c@{}}%
\phantom{0}\\\phantom{0}\\\phantom{0}
\end{tabular}\endgroup%
}}\right.$}%
\begingroup \smaller\smaller\smaller\begin{tabular}{@{}c@{}}%
-1/2\\\phantom{0}\\\phantom{0}
\end{tabular}\endgroup%
\kern3pt%
\begingroup \smaller\smaller\smaller\begin{tabular}{@{}c@{}}%
\phantom{0}\\15/2\\\phantom{0}
\end{tabular}\endgroup%
\kern3pt%
\begingroup \smaller\smaller\smaller\begin{tabular}{@{}c@{}}%
\phantom{0}\\\phantom{0}\\40
\end{tabular}\endgroup%
{$\left.\llap{\phantom{%
\begingroup \smaller\smaller\smaller\begin{tabular}{@{}c@{}}%
\phantom{0}\\\phantom{0}\\\phantom{0}
\end{tabular}\endgroup%
}}\!\right]$}%
{$\left[\!\llap{\phantom{%
\begingroup \smaller\smaller\smaller\begin{tabular}{@{}c@{}}%
0\\0\\0
\end{tabular}\endgroup%
}}\right.$}%
\begingroup \smaller\smaller\smaller\begin{tabular}{@{}c@{}}%
3\\1\\0
\end{tabular}\endgroup%
\kern3pt%
\begingroup \smaller\smaller\smaller\begin{tabular}{@{}c@{}}%
40\\8\\-3
\end{tabular}\endgroup%
\kern3pt%
\begingroup \smaller\smaller\smaller\begin{tabular}{@{}c@{}}%
30\\4\\-3
\end{tabular}\endgroup%
\kern3pt%
\begingroup \smaller\smaller\smaller\begin{tabular}{@{}c@{}}%
8\\0\\-1
\end{tabular}\endgroup%
{$\left.\llap{\phantom{%
\begingroup \smaller\smaller\smaller\begin{tabular}{@{}c@{}}%
0\\0\\0
\end{tabular}\endgroup%
}}\!\right]$}%
}%
\ifdim\wd\matricesbox>\halfwidth\myboxwidth=\hsize\else\myboxwidth=\halfwidth\fi
\vbox{%
\ifdim\myboxwidth=\hsize
\setbox\onelinebox=\hbox{%
\vbox{\hbox{%
$\Pi_{12,18}$ spans $L_{17.25}$%
}\hbox{%
$|222|222|222|222\rtimes D_{4}$%
}%
}%
\hfill\copy\matricesbox
}%
\ifdim\wd\onelinebox>\myboxwidth
\hbox to \myboxwidth{%
$\Pi_{12,18}$ spans $L_{17.25}$%
\hfil
$|222|222|222|222\rtimes D_{4}$%
}%
\box\matricesbox
\else
\hbox to \myboxwidth{%
\unhbox\onelinebox
}%
\fi
\else
\hbox to \myboxwidth{%
$\Pi_{12,18}$ spans $L_{17.25}$%
\hfil}%
\hbox to \myboxwidth{%
$|222|222|222|222\rtimes D_{4}$%
\hfil}%
\box\matricesbox
\fi
}%
\hfill\discretionary{}{}{}%
\setbox\matricesbox=\hbox{%
{$\left[\!\llap{\phantom{%
\begingroup \smaller\smaller\smaller\begin{tabular}{@{}c@{}}%
\phantom{0}\\\phantom{0}\\\phantom{0}
\end{tabular}\endgroup%
}}\right.$}%
\begingroup \smaller\smaller\smaller\begin{tabular}{@{}c@{}}%
-1\\\phantom{0}\\\phantom{0}
\end{tabular}\endgroup%
\kern3pt%
\begingroup \smaller\smaller\smaller\begin{tabular}{@{}c@{}}%
\phantom{0}\\12\\\phantom{0}
\end{tabular}\endgroup%
\kern3pt%
\begingroup \smaller\smaller\smaller\begin{tabular}{@{}c@{}}%
\phantom{0}\\\phantom{0}\\30
\end{tabular}\endgroup%
{$\left.\llap{\phantom{%
\begingroup \smaller\smaller\smaller\begin{tabular}{@{}c@{}}%
\phantom{0}\\\phantom{0}\\\phantom{0}
\end{tabular}\endgroup%
}}\!\right]$}%
{$\left[\!\llap{\phantom{%
\begingroup \smaller\smaller\smaller\begin{tabular}{@{}c@{}}%
0\\0\\0
\end{tabular}\endgroup%
}}\right.$}%
\begingroup \smaller\smaller\smaller\begin{tabular}{@{}c@{}}%
3\\-1\\0
\end{tabular}\endgroup%
\kern3pt%
\begingroup \smaller\smaller\smaller\begin{tabular}{@{}c@{}}%
20\\-5\\-2
\end{tabular}\endgroup%
\kern3pt%
\begingroup \smaller\smaller\smaller\begin{tabular}{@{}c@{}}%
6\\-1\\-1
\end{tabular}\endgroup%
\kern3pt%
\begingroup \smaller\smaller\smaller\begin{tabular}{@{}c@{}}%
5\\0\\-1
\end{tabular}\endgroup%
{$\left.\llap{\phantom{%
\begingroup \smaller\smaller\smaller\begin{tabular}{@{}c@{}}%
0\\0\\0
\end{tabular}\endgroup%
}}\!\right]$}%
}%
\ifdim\wd\matricesbox>\halfwidth\myboxwidth=\hsize\else\myboxwidth=\halfwidth\fi
\vbox{%
\ifdim\myboxwidth=\hsize
\setbox\onelinebox=\hbox{%
\vbox{\hbox{%
$\Pi_{12,19}$ spans $L_{205.10}$%
}\hbox{%
$|222|222|222|222\rtimes D_{4}$%
}%
}%
\hfill\copy\matricesbox
}%
\ifdim\wd\onelinebox>\myboxwidth
\hbox to \myboxwidth{%
$\Pi_{12,19}$ spans $L_{205.10}$%
\hfil
$|222|222|222|222\rtimes D_{4}$%
}%
\box\matricesbox
\else
\hbox to \myboxwidth{%
\unhbox\onelinebox
}%
\fi
\else
\hbox to \myboxwidth{%
$\Pi_{12,19}$ spans $L_{205.10}$%
\hfil}%
\hbox to \myboxwidth{%
$|222|222|222|222\rtimes D_{4}$%
\hfil}%
\box\matricesbox
\fi
}%
\hfill\discretionary{}{}{}%
\setbox\matricesbox=\hbox{%
{$\left[\!\llap{\phantom{%
\begingroup \smaller\smaller\smaller\begin{tabular}{@{}c@{}}%
\phantom{0}\\\phantom{0}\\\phantom{0}
\end{tabular}\endgroup%
}}\right.$}%
\begingroup \smaller\smaller\smaller\begin{tabular}{@{}c@{}}%
-1/2\\\phantom{0}\\\phantom{0}
\end{tabular}\endgroup%
\kern3pt%
\begingroup \smaller\smaller\smaller\begin{tabular}{@{}c@{}}%
\phantom{0}\\3/2\\\phantom{0}
\end{tabular}\endgroup%
\kern3pt%
\begingroup \smaller\smaller\smaller\begin{tabular}{@{}c@{}}%
\phantom{0}\\\phantom{0}\\10
\end{tabular}\endgroup%
{$\left.\llap{\phantom{%
\begingroup \smaller\smaller\smaller\begin{tabular}{@{}c@{}}%
\phantom{0}\\\phantom{0}\\\phantom{0}
\end{tabular}\endgroup%
}}\!\right]$}%
{$\left[\!\llap{\phantom{%
\begingroup \smaller\smaller\smaller\begin{tabular}{@{}c@{}}%
0\\0\\0
\end{tabular}\endgroup%
}}\right.$}%
\begingroup \smaller\smaller\smaller\begin{tabular}{@{}c@{}}%
6\\4\\0
\end{tabular}\endgroup%
\kern3pt%
\begingroup \smaller\smaller\smaller\begin{tabular}{@{}c@{}}%
32\\16\\-4
\end{tabular}\endgroup%
\kern3pt%
\begingroup \smaller\smaller\smaller\begin{tabular}{@{}c@{}}%
15\\5\\-3
\end{tabular}\endgroup%
\kern3pt%
\begingroup \smaller\smaller\smaller\begin{tabular}{@{}c@{}}%
8\\0\\-2
\end{tabular}\endgroup%
{$\left.\llap{\phantom{%
\begingroup \smaller\smaller\smaller\begin{tabular}{@{}c@{}}%
0\\0\\0
\end{tabular}\endgroup%
}}\!\right]$}%
}%
\ifdim\wd\matricesbox>\halfwidth\myboxwidth=\hsize\else\myboxwidth=\halfwidth\fi
\vbox{%
\ifdim\myboxwidth=\hsize
\setbox\onelinebox=\hbox{%
\vbox{\hbox{%
$\Pi_{12,20}$ spans $L_{17.3}$%
}\hbox{%
$|222|222|222|222\rtimes D_{4}$%
}%
}%
\hfill\copy\matricesbox
}%
\ifdim\wd\onelinebox>\myboxwidth
\hbox to \myboxwidth{%
$\Pi_{12,20}$ spans $L_{17.3}$%
\hfil
$|222|222|222|222\rtimes D_{4}$%
}%
\box\matricesbox
\else
\hbox to \myboxwidth{%
\unhbox\onelinebox
}%
\fi
\else
\hbox to \myboxwidth{%
$\Pi_{12,20}$ spans $L_{17.3}$%
\hfil}%
\hbox to \myboxwidth{%
$|222|222|222|222\rtimes D_{4}$%
\hfil}%
\box\matricesbox
\fi
}%
\hfill\discretionary{}{}{}%
\setbox\matricesbox=\hbox{%
{$\left[\!\llap{\phantom{%
\begingroup \smaller\smaller\smaller\begin{tabular}{@{}c@{}}%
\phantom{0}\\\phantom{0}\\\phantom{0}
\end{tabular}\endgroup%
}}\right.$}%
\begingroup \smaller\smaller\smaller\begin{tabular}{@{}c@{}}%
-1/8\\\phantom{0}\\\phantom{0}
\end{tabular}\endgroup%
\kern3pt%
\begingroup \smaller\smaller\smaller\begin{tabular}{@{}c@{}}%
\phantom{0}\\35/2\\\phantom{0}
\end{tabular}\endgroup%
\kern3pt%
\begingroup \smaller\smaller\smaller\begin{tabular}{@{}c@{}}%
\phantom{0}\\\phantom{0}\\21
\end{tabular}\endgroup%
{$\left.\llap{\phantom{%
\begingroup \smaller\smaller\smaller\begin{tabular}{@{}c@{}}%
\phantom{0}\\\phantom{0}\\\phantom{0}
\end{tabular}\endgroup%
}}\!\right]$}%
{$\left[\!\llap{\phantom{%
\begingroup \smaller\smaller\smaller\begin{tabular}{@{}c@{}}%
0\\0\\0
\end{tabular}\endgroup%
}}\right.$}%
\begingroup \smaller\smaller\smaller\begin{tabular}{@{}c@{}}%
20\\-2\\0
\end{tabular}\endgroup%
\kern3pt%
\begingroup \smaller\smaller\smaller\begin{tabular}{@{}c@{}}%
84\\-6\\4
\end{tabular}\endgroup%
\kern3pt%
\begingroup \smaller\smaller\smaller\begin{tabular}{@{}c@{}}%
70\\-3\\5
\end{tabular}\endgroup%
\kern3pt%
\begingroup \smaller\smaller\smaller\begin{tabular}{@{}c@{}}%
48\\0\\4
\end{tabular}\endgroup%
{$\left.\llap{\phantom{%
\begingroup \smaller\smaller\smaller\begin{tabular}{@{}c@{}}%
0\\0\\0
\end{tabular}\endgroup%
}}\!\right]$}%
}%
\ifdim\wd\matricesbox>\halfwidth\myboxwidth=\hsize\else\myboxwidth=\halfwidth\fi
\vbox{%
\ifdim\myboxwidth=\hsize
\setbox\onelinebox=\hbox{%
\vbox{\hbox{%
$\Pi_{12,21}$ spans $L_{69.11}$%
}\hbox{%
$|222|222|222|222\rtimes D_{4}$%
}%
}%
\hfill\copy\matricesbox
}%
\ifdim\wd\onelinebox>\myboxwidth
\hbox to \myboxwidth{%
$\Pi_{12,21}$ spans $L_{69.11}$%
\hfil
$|222|222|222|222\rtimes D_{4}$%
}%
\box\matricesbox
\else
\hbox to \myboxwidth{%
\unhbox\onelinebox
}%
\fi
\else
\hbox to \myboxwidth{%
$\Pi_{12,21}$ spans $L_{69.11}$%
\hfil}%
\hbox to \myboxwidth{%
$|222|222|222|222\rtimes D_{4}$%
\hfil}%
\box\matricesbox
\fi
}%
\hfill\discretionary{}{}{}%
\setbox\matricesbox=\hbox{%
{$\left[\!\llap{\phantom{%
\begingroup \smaller\smaller\smaller\begin{tabular}{@{}c@{}}%
\phantom{0}\\\phantom{0}\\\phantom{0}
\end{tabular}\endgroup%
}}\right.$}%
\begingroup \smaller\smaller\smaller\begin{tabular}{@{}c@{}}%
-1\\\phantom{0}\\\phantom{0}
\end{tabular}\endgroup%
\kern3pt%
\begingroup \smaller\smaller\smaller\begin{tabular}{@{}c@{}}%
\phantom{0}\\45/2\\\phantom{0}
\end{tabular}\endgroup%
\kern3pt%
\begingroup \smaller\smaller\smaller\begin{tabular}{@{}c@{}}%
\phantom{0}\\\phantom{0}\\3/2
\end{tabular}\endgroup%
{$\left.\llap{\phantom{%
\begingroup \smaller\smaller\smaller\begin{tabular}{@{}c@{}}%
\phantom{0}\\\phantom{0}\\\phantom{0}
\end{tabular}\endgroup%
}}\!\right]$}%
{$\left[\!\llap{\phantom{%
\begingroup \smaller\smaller\smaller\begin{tabular}{@{}c@{}}%
0\\0\\0
\end{tabular}\endgroup%
}}\right.$}%
\begingroup \smaller\smaller\smaller\begin{tabular}{@{}c@{}}%
9\\-2\\0
\end{tabular}\endgroup%
\kern3pt%
\begingroup \smaller\smaller\smaller\begin{tabular}{@{}c@{}}%
15\\-3\\-5
\end{tabular}\endgroup%
\kern3pt%
\begingroup \smaller\smaller\smaller\begin{tabular}{@{}c@{}}%
15\\-2\\-10
\end{tabular}\endgroup%
\kern3pt%
\begingroup \smaller\smaller\smaller\begin{tabular}{@{}c@{}}%
2\\0\\-2
\end{tabular}\endgroup%
{$\left.\llap{\phantom{%
\begingroup \smaller\smaller\smaller\begin{tabular}{@{}c@{}}%
0\\0\\0
\end{tabular}\endgroup%
}}\!\right]$}%
}%
\ifdim\wd\matricesbox>\halfwidth\myboxwidth=\hsize\else\myboxwidth=\halfwidth\fi
\vbox{%
\ifdim\myboxwidth=\hsize
\setbox\onelinebox=\hbox{%
\vbox{\hbox{%
$\Pi_{12,22}$ spans $L_{18.17}$%
}\hbox{%
$\infty2|2\infty2|2\infty2|2\infty2|2\rtimes D_{4}$%
}%
}%
\hfill\copy\matricesbox
}%
\ifdim\wd\onelinebox>\myboxwidth
\hbox to \myboxwidth{%
$\Pi_{12,22}$ spans $L_{18.17}$%
\hfil
$\infty2|2\infty2|2\infty2|2\infty2|2\rtimes D_{4}$%
}%
\box\matricesbox
\else
\hbox to \myboxwidth{%
\unhbox\onelinebox
}%
\fi
\else
\hbox to \myboxwidth{%
$\Pi_{12,22}$ spans $L_{18.17}$%
\hfil}%
\hbox to \myboxwidth{%
$\infty2|2\infty2|2\infty2|2\infty2|2\rtimes D_{4}$%
\hfil}%
\box\matricesbox
\fi
}%
\hfill\discretionary{}{}{}%
\setbox\matricesbox=\hbox{%
{$\left[\!\llap{\phantom{%
\begingroup \smaller\smaller\smaller\begin{tabular}{@{}c@{}}%
\phantom{0}\\\phantom{0}\\\phantom{0}
\end{tabular}\endgroup%
}}\right.$}%
\begingroup \smaller\smaller\smaller\begin{tabular}{@{}c@{}}%
-1\\\phantom{0}\\\phantom{0}
\end{tabular}\endgroup%
\kern3pt%
\begingroup \smaller\smaller\smaller\begin{tabular}{@{}c@{}}%
\phantom{0}\\6\\\phantom{0}
\end{tabular}\endgroup%
\kern3pt%
\begingroup \smaller\smaller\smaller\begin{tabular}{@{}c@{}}%
\phantom{0}\\\phantom{0}\\6
\end{tabular}\endgroup%
{$\left.\llap{\phantom{%
\begingroup \smaller\smaller\smaller\begin{tabular}{@{}c@{}}%
\phantom{0}\\\phantom{0}\\\phantom{0}
\end{tabular}\endgroup%
}}\!\right]$}%
{$\left[\!\llap{\phantom{%
\begingroup \smaller\smaller\smaller\begin{tabular}{@{}c@{}}%
0\\0\\0
\end{tabular}\endgroup%
}}\right.$}%
\begingroup \smaller\smaller\smaller\begin{tabular}{@{}c@{}}%
2\\-1\\0
\end{tabular}\endgroup%
\kern3pt%
\begingroup \smaller\smaller\smaller\begin{tabular}{@{}c@{}}%
24\\-8\\-6
\end{tabular}\endgroup%
\kern3pt%
\begingroup \smaller\smaller\smaller\begin{tabular}{@{}c@{}}%
24\\-6\\-8
\end{tabular}\endgroup%
\kern3pt%
\begingroup \smaller\smaller\smaller\begin{tabular}{@{}c@{}}%
12\\-1\\-5
\end{tabular}\endgroup%
\kern3pt%
\begingroup \smaller\smaller\smaller\begin{tabular}{@{}c@{}}%
12\\1\\-5
\end{tabular}\endgroup%
\kern3pt%
\begingroup \smaller\smaller\smaller\begin{tabular}{@{}c@{}}%
3\\1\\-1
\end{tabular}\endgroup%
\kern3pt%
\begingroup \smaller\smaller\smaller\begin{tabular}{@{}c@{}}%
2\\1\\0
\end{tabular}\endgroup%
{$\left.\llap{\phantom{%
\begingroup \smaller\smaller\smaller\begin{tabular}{@{}c@{}}%
0\\0\\0
\end{tabular}\endgroup%
}}\!\right]$}%
}%
\ifdim\wd\matricesbox>\halfwidth\myboxwidth=\hsize\else\myboxwidth=\halfwidth\fi
\vbox{%
\ifdim\myboxwidth=\hsize
\setbox\onelinebox=\hbox{%
\vbox{\hbox{%
$\Pi_{12,23}$ spans $L_{123.8}$%
}\hbox{%
$|224222|222422\rtimes D_{2}$%
}%
}%
\hfill\copy\matricesbox
}%
\ifdim\wd\onelinebox>\myboxwidth
\hbox to \myboxwidth{%
$\Pi_{12,23}$ spans $L_{123.8}$%
\hfil
$|224222|222422\rtimes D_{2}$%
}%
\box\matricesbox
\else
\hbox to \myboxwidth{%
\unhbox\onelinebox
}%
\fi
\else
\hbox to \myboxwidth{%
$\Pi_{12,23}$ spans $L_{123.8}$%
\hfil}%
\hbox to \myboxwidth{%
$|224222|222422\rtimes D_{2}$%
\hfil}%
\box\matricesbox
\fi
}%
\hfill\discretionary{}{}{}%
\setbox\matricesbox=\hbox{%
{$\left[\!\llap{\phantom{%
\begingroup \smaller\smaller\smaller\begin{tabular}{@{}c@{}}%
\phantom{0}\\\phantom{0}\\\phantom{0}
\end{tabular}\endgroup%
}}\right.$}%
\begingroup \smaller\smaller\smaller\begin{tabular}{@{}c@{}}%
-1\\\phantom{0}\\\phantom{0}
\end{tabular}\endgroup%
\kern3pt%
\begingroup \smaller\smaller\smaller\begin{tabular}{@{}c@{}}%
\phantom{0}\\3\\\phantom{0}
\end{tabular}\endgroup%
\kern3pt%
\begingroup \smaller\smaller\smaller\begin{tabular}{@{}c@{}}%
\phantom{0}\\\phantom{0}\\3
\end{tabular}\endgroup%
{$\left.\llap{\phantom{%
\begingroup \smaller\smaller\smaller\begin{tabular}{@{}c@{}}%
\phantom{0}\\\phantom{0}\\\phantom{0}
\end{tabular}\endgroup%
}}\!\right]$}%
{$\left[\!\llap{\phantom{%
\begingroup \smaller\smaller\smaller\begin{tabular}{@{}c@{}}%
0\\0\\0
\end{tabular}\endgroup%
}}\right.$}%
\begingroup \smaller\smaller\smaller\begin{tabular}{@{}c@{}}%
3\\-2\\0
\end{tabular}\endgroup%
\kern3pt%
\begingroup \smaller\smaller\smaller\begin{tabular}{@{}c@{}}%
12\\-6\\4
\end{tabular}\endgroup%
\kern3pt%
\begingroup \smaller\smaller\smaller\begin{tabular}{@{}c@{}}%
12\\-4\\6
\end{tabular}\endgroup%
\kern3pt%
\begingroup \smaller\smaller\smaller\begin{tabular}{@{}c@{}}%
24\\-2\\14
\end{tabular}\endgroup%
\kern3pt%
\begingroup \smaller\smaller\smaller\begin{tabular}{@{}c@{}}%
24\\2\\14
\end{tabular}\endgroup%
\kern3pt%
\begingroup \smaller\smaller\smaller\begin{tabular}{@{}c@{}}%
2\\1\\1
\end{tabular}\endgroup%
\kern3pt%
\begingroup \smaller\smaller\smaller\begin{tabular}{@{}c@{}}%
3\\2\\0
\end{tabular}\endgroup%
{$\left.\llap{\phantom{%
\begingroup \smaller\smaller\smaller\begin{tabular}{@{}c@{}}%
0\\0\\0
\end{tabular}\endgroup%
}}\!\right]$}%
}%
\ifdim\wd\matricesbox>\halfwidth\myboxwidth=\hsize\else\myboxwidth=\halfwidth\fi
\vbox{%
\ifdim\myboxwidth=\hsize
\setbox\onelinebox=\hbox{%
\vbox{\hbox{%
$\Pi_{12,24}$ spans $L_{123.8}$%
}\hbox{%
$22422|224222|2\rtimes D_{2}$%
}%
}%
\hfill\copy\matricesbox
}%
\ifdim\wd\onelinebox>\myboxwidth
\hbox to \myboxwidth{%
$\Pi_{12,24}$ spans $L_{123.8}$%
\hfil
$22422|224222|2\rtimes D_{2}$%
}%
\box\matricesbox
\else
\hbox to \myboxwidth{%
\unhbox\onelinebox
}%
\fi
\else
\hbox to \myboxwidth{%
$\Pi_{12,24}$ spans $L_{123.8}$%
\hfil}%
\hbox to \myboxwidth{%
$22422|224222|2\rtimes D_{2}$%
\hfil}%
\box\matricesbox
\fi
}%
\hfill\discretionary{}{}{}%
\setbox\matricesbox=\hbox{%
{$\left[\!\llap{\phantom{%
\begingroup \smaller\smaller\smaller\begin{tabular}{@{}c@{}}%
\phantom{0}\\\phantom{0}\\\phantom{0}
\end{tabular}\endgroup%
}}\right.$}%
\begingroup \smaller\smaller\smaller\begin{tabular}{@{}c@{}}%
-1\\\phantom{0}\\\phantom{0}
\end{tabular}\endgroup%
\kern3pt%
\begingroup \smaller\smaller\smaller\begin{tabular}{@{}c@{}}%
\phantom{0}\\6\\\phantom{0}
\end{tabular}\endgroup%
\kern3pt%
\begingroup \smaller\smaller\smaller\begin{tabular}{@{}c@{}}%
\phantom{0}\\\phantom{0}\\6
\end{tabular}\endgroup%
{$\left.\llap{\phantom{%
\begingroup \smaller\smaller\smaller\begin{tabular}{@{}c@{}}%
\phantom{0}\\\phantom{0}\\\phantom{0}
\end{tabular}\endgroup%
}}\!\right]$}%
{$\left[\!\llap{\phantom{%
\begingroup \smaller\smaller\smaller\begin{tabular}{@{}c@{}}%
0\\0\\0
\end{tabular}\endgroup%
}}\right.$}%
\begingroup \smaller\smaller\smaller\begin{tabular}{@{}c@{}}%
12\\5\\-1
\end{tabular}\endgroup%
\kern3pt%
\begingroup \smaller\smaller\smaller\begin{tabular}{@{}c@{}}%
24\\8\\-6
\end{tabular}\endgroup%
\kern3pt%
\begingroup \smaller\smaller\smaller\begin{tabular}{@{}c@{}}%
24\\6\\-8
\end{tabular}\endgroup%
\kern3pt%
\begingroup \smaller\smaller\smaller\begin{tabular}{@{}c@{}}%
2\\0\\-1
\end{tabular}\endgroup%
\kern3pt%
\begingroup \smaller\smaller\smaller\begin{tabular}{@{}c@{}}%
3\\-1\\-1
\end{tabular}\endgroup%
\kern3pt%
\begingroup \smaller\smaller\smaller\begin{tabular}{@{}c@{}}%
12\\-5\\-1
\end{tabular}\endgroup%
{$\left.\llap{\phantom{%
\begingroup \smaller\smaller\smaller\begin{tabular}{@{}c@{}}%
0\\0\\0
\end{tabular}\endgroup%
}}\!\right]$}%
}%
\ifdim\wd\matricesbox>\halfwidth\myboxwidth=\hsize\else\myboxwidth=\halfwidth\fi
\vbox{%
\ifdim\myboxwidth=\hsize
\setbox\onelinebox=\hbox{%
\vbox{\hbox{%
$\Pi_{12,25}$ spans $L_{123.8}$%
}\hbox{%
$224\slashtwo42222\slashtwo22\rtimes D_{2}$%
}%
}%
\hfill\copy\matricesbox
}%
\ifdim\wd\onelinebox>\myboxwidth
\hbox to \myboxwidth{%
$\Pi_{12,25}$ spans $L_{123.8}$%
\hfil
$224\slashtwo42222\slashtwo22\rtimes D_{2}$%
}%
\box\matricesbox
\else
\hbox to \myboxwidth{%
\unhbox\onelinebox
}%
\fi
\else
\hbox to \myboxwidth{%
$\Pi_{12,25}$ spans $L_{123.8}$%
\hfil}%
\hbox to \myboxwidth{%
$224\slashtwo42222\slashtwo22\rtimes D_{2}$%
\hfil}%
\box\matricesbox
\fi
}%
\hfill\discretionary{}{}{}%
\setbox\matricesbox=\hbox{%
{$\left[\!\llap{\phantom{%
\begingroup \smaller\smaller\smaller\begin{tabular}{@{}c@{}}%
\phantom{0}\\\phantom{0}\\\phantom{0}
\end{tabular}\endgroup%
}}\right.$}%
\begingroup \smaller\smaller\smaller\begin{tabular}{@{}c@{}}%
-1\\\phantom{0}\\\phantom{0}
\end{tabular}\endgroup%
\kern3pt%
\begingroup \smaller\smaller\smaller\begin{tabular}{@{}c@{}}%
\phantom{0}\\3\\\phantom{0}
\end{tabular}\endgroup%
\kern3pt%
\begingroup \smaller\smaller\smaller\begin{tabular}{@{}c@{}}%
\phantom{0}\\\phantom{0}\\3
\end{tabular}\endgroup%
{$\left.\llap{\phantom{%
\begingroup \smaller\smaller\smaller\begin{tabular}{@{}c@{}}%
\phantom{0}\\\phantom{0}\\\phantom{0}
\end{tabular}\endgroup%
}}\!\right]$}%
{$\left[\!\llap{\phantom{%
\begingroup \smaller\smaller\smaller\begin{tabular}{@{}c@{}}%
0\\0\\0
\end{tabular}\endgroup%
}}\right.$}%
\begingroup \smaller\smaller\smaller\begin{tabular}{@{}c@{}}%
24\\14\\-2
\end{tabular}\endgroup%
\kern3pt%
\begingroup \smaller\smaller\smaller\begin{tabular}{@{}c@{}}%
12\\6\\-4
\end{tabular}\endgroup%
\kern3pt%
\begingroup \smaller\smaller\smaller\begin{tabular}{@{}c@{}}%
12\\4\\-6
\end{tabular}\endgroup%
\kern3pt%
\begingroup \smaller\smaller\smaller\begin{tabular}{@{}c@{}}%
3\\0\\-2
\end{tabular}\endgroup%
\kern3pt%
\begingroup \smaller\smaller\smaller\begin{tabular}{@{}c@{}}%
2\\-1\\-1
\end{tabular}\endgroup%
\kern3pt%
\begingroup \smaller\smaller\smaller\begin{tabular}{@{}c@{}}%
24\\-14\\-2
\end{tabular}\endgroup%
{$\left.\llap{\phantom{%
\begingroup \smaller\smaller\smaller\begin{tabular}{@{}c@{}}%
0\\0\\0
\end{tabular}\endgroup%
}}\!\right]$}%
}%
\ifdim\wd\matricesbox>\halfwidth\myboxwidth=\hsize\else\myboxwidth=\halfwidth\fi
\vbox{%
\ifdim\myboxwidth=\hsize
\setbox\onelinebox=\hbox{%
\vbox{\hbox{%
$\Pi_{12,26}$ spans $L_{123.8}$%
}\hbox{%
$224\slashtwo42222\slashtwo22\rtimes D_{2}$%
}%
}%
\hfill\copy\matricesbox
}%
\ifdim\wd\onelinebox>\myboxwidth
\hbox to \myboxwidth{%
$\Pi_{12,26}$ spans $L_{123.8}$%
\hfil
$224\slashtwo42222\slashtwo22\rtimes D_{2}$%
}%
\box\matricesbox
\else
\hbox to \myboxwidth{%
\unhbox\onelinebox
}%
\fi
\else
\hbox to \myboxwidth{%
$\Pi_{12,26}$ spans $L_{123.8}$%
\hfil}%
\hbox to \myboxwidth{%
$224\slashtwo42222\slashtwo22\rtimes D_{2}$%
\hfil}%
\box\matricesbox
\fi
}%
\hfill\discretionary{}{}{}%
\setbox\matricesbox=\hbox{%
{$\left[\!\llap{\phantom{%
\begingroup \smaller\smaller\smaller\begin{tabular}{@{}c@{}}%
\phantom{0}\\\phantom{0}\\\phantom{0}
\end{tabular}\endgroup%
}}\right.$}%
\begingroup \smaller\smaller\smaller\begin{tabular}{@{}c@{}}%
-1\\\phantom{0}\\\phantom{0}
\end{tabular}\endgroup%
\kern3pt%
\begingroup \smaller\smaller\smaller\begin{tabular}{@{}c@{}}%
\phantom{0}\\6\\\phantom{0}
\end{tabular}\endgroup%
\kern3pt%
\begingroup \smaller\smaller\smaller\begin{tabular}{@{}c@{}}%
\phantom{0}\\\phantom{0}\\6
\end{tabular}\endgroup%
{$\left.\llap{\phantom{%
\begingroup \smaller\smaller\smaller\begin{tabular}{@{}c@{}}%
\phantom{0}\\\phantom{0}\\\phantom{0}
\end{tabular}\endgroup%
}}\!\right]$}%
{$\left[\!\llap{\phantom{%
\begingroup \smaller\smaller\smaller\begin{tabular}{@{}c@{}}%
0\\0\\0
\end{tabular}\endgroup%
}}\right.$}%
\begingroup \smaller\smaller\smaller\begin{tabular}{@{}c@{}}%
2\\1\\0
\end{tabular}\endgroup%
\kern3pt%
\begingroup \smaller\smaller\smaller\begin{tabular}{@{}c@{}}%
24\\8\\6
\end{tabular}\endgroup%
\kern3pt%
\begingroup \smaller\smaller\smaller\begin{tabular}{@{}c@{}}%
24\\6\\8
\end{tabular}\endgroup%
\kern3pt%
\begingroup \smaller\smaller\smaller\begin{tabular}{@{}c@{}}%
12\\1\\5
\end{tabular}\endgroup%
\kern3pt%
\begingroup \smaller\smaller\smaller\begin{tabular}{@{}c@{}}%
12\\-1\\5
\end{tabular}\endgroup%
\kern3pt%
\begingroup \smaller\smaller\smaller\begin{tabular}{@{}c@{}}%
3\\-1\\1
\end{tabular}\endgroup%
{$\left.\llap{\phantom{%
\begingroup \smaller\smaller\smaller\begin{tabular}{@{}c@{}}%
0\\0\\0
\end{tabular}\endgroup%
}}\!\right]$}%
}%
\ifdim\wd\matricesbox>\halfwidth\myboxwidth=\hsize\else\myboxwidth=\halfwidth\fi
\vbox{%
\ifdim\myboxwidth=\hsize
\setbox\onelinebox=\hbox{%
\vbox{\hbox{%
$\Pi_{12,27}$ spans $L_{123.8}$%
}\hbox{%
$224222224222\rtimes C_{2}$%
}%
}%
\hfill\copy\matricesbox
}%
\ifdim\wd\onelinebox>\myboxwidth
\hbox to \myboxwidth{%
$\Pi_{12,27}$ spans $L_{123.8}$%
\hfil
$224222224222\rtimes C_{2}$%
}%
\box\matricesbox
\else
\hbox to \myboxwidth{%
\unhbox\onelinebox
}%
\fi
\else
\hbox to \myboxwidth{%
$\Pi_{12,27}$ spans $L_{123.8}$%
\hfil}%
\hbox to \myboxwidth{%
$224222224222\rtimes C_{2}$%
\hfil}%
\box\matricesbox
\fi
}%
\hfill\discretionary{}{}{}%
\setbox\matricesbox=\hbox{%
{$\left[\!\llap{\phantom{%
\begingroup \smaller\smaller\smaller
\begin{tabular}{@{}c@{}}%
\phantom{0}\\\phantom{0}\\\phantom{0}
\end{tabular}\endgroup%
}}\right.$}%
\begingroup \smaller\smaller\smaller\begin{tabular}{@{}c@{}}%
-1\\\phantom{0}\\\phantom{0}
\end{tabular}\endgroup%
\kern3pt%
\begingroup \smaller\smaller\smaller\begin{tabular}{@{}c@{}}%
\phantom{0}\\6\\\phantom{0}
\end{tabular}\endgroup%
\kern3pt%
\begingroup \smaller\smaller\smaller\begin{tabular}{@{}c@{}}%
\phantom{0}\\\phantom{0}\\6
\end{tabular}\endgroup%
{$\left.\llap{\phantom{%
\begingroup \smaller\smaller\smaller\begin{tabular}{@{}c@{}}%
\phantom{0}\\\phantom{0}\\\phantom{0}
\end{tabular}\endgroup%
}}\!\right]$}%
{$\left[\!\llap{\phantom{%
\begingroup \smaller\smaller\smaller\begin{tabular}{@{}c@{}}%
0\\0\\0
\end{tabular}\endgroup%
}}\right.$}%
\begingroup \smaller\smaller\smaller\begin{tabular}{@{}c@{}}%
2\\1\\0
\end{tabular}\endgroup%
\kern3pt%
\begingroup \smaller\smaller\smaller\begin{tabular}{@{}c@{}}%
24\\8\\6
\end{tabular}\endgroup%
\kern3pt%
\begingroup \smaller\smaller\smaller\begin{tabular}{@{}c@{}}%
24\\6\\8
\end{tabular}\endgroup%
\kern3pt%
\begingroup \smaller\smaller\smaller\begin{tabular}{@{}c@{}}%
12\\1\\5
\end{tabular}\endgroup%
\kern3pt%
\begingroup \smaller\smaller\smaller\begin{tabular}{@{}c@{}}%
12\\-1\\5
\end{tabular}\endgroup%
\kern3pt%
\begingroup \smaller\smaller\smaller\begin{tabular}{@{}c@{}}%
3\\-1\\1
\end{tabular}\endgroup%
\kern3pt%
\begingroup \smaller\smaller\smaller\begin{tabular}{@{}c@{}}%
2\\-1\\0
\end{tabular}\endgroup%
\kern3pt%
\begingroup \smaller\smaller\smaller\begin{tabular}{@{}c@{}}%
24\\-8\\-6
\end{tabular}\endgroup%
\kern3pt%
\begingroup \smaller\smaller\smaller\begin{tabular}{@{}c@{}}%
24\\-6\\-8
\end{tabular}\endgroup%
\kern3pt%
\begingroup \smaller\smaller\smaller\begin{tabular}{@{}c@{}}%
2\\0\\-1
\end{tabular}\endgroup%
\kern3pt%
\begingroup \smaller\smaller\smaller\begin{tabular}{@{}c@{}}%
24\\6\\-8
\end{tabular}\endgroup%
\kern3pt%
\begingroup \smaller\smaller\smaller\begin{tabular}{@{}c@{}}%
24\\8\\-6
\end{tabular}\endgroup%
{$\left.\llap{\phantom{%
\begingroup \smaller\smaller\smaller\begin{tabular}{@{}c@{}}%
0\\0\\0
\end{tabular}\endgroup%
}}\!\right]$}%
}%
\ifdim\wd\matricesbox>\halfwidth\myboxwidth=\hsize\else\myboxwidth=\halfwidth\fi
\vbox{%
\ifdim\myboxwidth=\hsize
\setbox\onelinebox=\hbox{%
\vbox{\hbox{%
$\Pi_{12,28}$ spans $L_{123.8}$%
}\hbox{%
$224222222222$%
}%
}%
\hfill\copy\matricesbox
}%
\ifdim\wd\onelinebox>\myboxwidth
\hbox to \myboxwidth{%
$\Pi_{12,28}$ spans $L_{123.8}$%
\hfil
$224222222222$%
}%
\box\matricesbox
\else
\hbox to \myboxwidth{%
\unhbox\onelinebox
}%
\fi
\else
\hbox to \myboxwidth{%
$\Pi_{12,28}$ spans $L_{123.8}$%
\hfil}%
\hbox to \myboxwidth{%
$224222222222$%
\hfil}%
\box\matricesbox
\fi
}%
\hfill\discretionary{}{}{}%
\setbox\matricesbox=\hbox{%
{$\left[\!\llap{\phantom{%
\begingroup \smaller\smaller\smaller
\begin{tabular}{@{}c@{}}%
\phantom{0}\\\phantom{0}\\\phantom{0}
\end{tabular}\endgroup%
}}\right.$}%
\begingroup \smaller\smaller\smaller\begin{tabular}{@{}c@{}}%
-1\\\phantom{0}\\\phantom{0}
\end{tabular}\endgroup%
\kern3pt%
\begingroup \smaller\smaller\smaller\begin{tabular}{@{}c@{}}%
\phantom{0}\\6\\\phantom{0}
\end{tabular}\endgroup%
\kern3pt%
\begingroup \smaller\smaller\smaller\begin{tabular}{@{}c@{}}%
\phantom{0}\\\phantom{0}\\6
\end{tabular}\endgroup%
{$\left.\llap{\phantom{%
\begingroup \smaller\smaller\smaller\begin{tabular}{@{}c@{}}%
\phantom{0}\\\phantom{0}\\\phantom{0}
\end{tabular}\endgroup%
}}\!\right]$}%
{$\left[\!\llap{\phantom{%
\begingroup \smaller\smaller\smaller\begin{tabular}{@{}c@{}}%
0\\0\\0
\end{tabular}\endgroup%
}}\right.$}%
\begingroup \smaller\smaller\smaller\begin{tabular}{@{}c@{}}%
2\\1\\0
\end{tabular}\endgroup%
\kern3pt%
\begingroup \smaller\smaller\smaller\begin{tabular}{@{}c@{}}%
24\\8\\6
\end{tabular}\endgroup%
\kern3pt%
\begingroup \smaller\smaller\smaller\begin{tabular}{@{}c@{}}%
24\\6\\8
\end{tabular}\endgroup%
\kern3pt%
\begingroup \smaller\smaller\smaller\begin{tabular}{@{}c@{}}%
12\\1\\5
\end{tabular}\endgroup%
\kern3pt%
\begingroup \smaller\smaller\smaller\begin{tabular}{@{}c@{}}%
12\\-1\\5
\end{tabular}\endgroup%
\kern3pt%
\begingroup \smaller\smaller\smaller\begin{tabular}{@{}c@{}}%
3\\-1\\1
\end{tabular}\endgroup%
\kern3pt%
\begingroup \smaller\smaller\smaller\begin{tabular}{@{}c@{}}%
12\\-5\\1
\end{tabular}\endgroup%
\kern3pt%
\begingroup \smaller\smaller\smaller\begin{tabular}{@{}c@{}}%
12\\-5\\-1
\end{tabular}\endgroup%
\kern3pt%
\begingroup \smaller\smaller\smaller\begin{tabular}{@{}c@{}}%
3\\-1\\-1
\end{tabular}\endgroup%
\kern3pt%
\begingroup \smaller\smaller\smaller\begin{tabular}{@{}c@{}}%
2\\0\\-1
\end{tabular}\endgroup%
\kern3pt%
\begingroup \smaller\smaller\smaller\begin{tabular}{@{}c@{}}%
24\\6\\-8
\end{tabular}\endgroup%
\kern3pt%
\begingroup \smaller\smaller\smaller\begin{tabular}{@{}c@{}}%
24\\8\\-6
\end{tabular}\endgroup%
{$\left.\llap{\phantom{%
\begingroup \smaller\smaller\smaller\begin{tabular}{@{}c@{}}%
0\\0\\0
\end{tabular}\endgroup%
}}\!\right]$}%
}%
\ifdim\wd\matricesbox>\halfwidth\myboxwidth=\hsize\else\myboxwidth=\halfwidth\fi
\vbox{%
\ifdim\myboxwidth=\hsize
\setbox\onelinebox=\hbox{%
\vbox{\hbox{%
$\Pi_{12,29}$ spans $L_{123.8}$%
}\hbox{%
$224222222222$%
}%
}%
\hfill\copy\matricesbox
}%
\ifdim\wd\onelinebox>\myboxwidth
\hbox to \myboxwidth{%
$\Pi_{12,29}$ spans $L_{123.8}$%
\hfil
$224222222222$%
}%
\box\matricesbox
\else
\hbox to \myboxwidth{%
\unhbox\onelinebox
}%
\fi
\else
\hbox to \myboxwidth{%
$\Pi_{12,29}$ spans $L_{123.8}$%
\hfil}%
\hbox to \myboxwidth{%
$224222222222$%
\hfil}%
\box\matricesbox
\fi
}%
\hfill\discretionary{}{}{}%
\setbox\matricesbox=\hbox{%
{$\left[\!\llap{\phantom{%
\begingroup \smaller\smaller\smaller
\begin{tabular}{@{}c@{}}%
\phantom{0}\\\phantom{0}\\\phantom{0}
\end{tabular}\endgroup%
}}\right.$}%
\begingroup \smaller\smaller\smaller\begin{tabular}{@{}c@{}}%
-1\\\phantom{0}\\\phantom{0}
\end{tabular}\endgroup%
\kern3pt%
\begingroup \smaller\smaller\smaller\begin{tabular}{@{}c@{}}%
\phantom{0}\\6\\\phantom{0}
\end{tabular}\endgroup%
\kern3pt%
\begingroup \smaller\smaller\smaller\begin{tabular}{@{}c@{}}%
\phantom{0}\\\phantom{0}\\6
\end{tabular}\endgroup%
{$\left.\llap{\phantom{%
\begingroup \smaller\smaller\smaller\begin{tabular}{@{}c@{}}%
\phantom{0}\\\phantom{0}\\\phantom{0}
\end{tabular}\endgroup%
}}\!\right]$}%
{$\left[\!\llap{\phantom{%
\begingroup \smaller\smaller\smaller\begin{tabular}{@{}c@{}}%
0\\0\\0
\end{tabular}\endgroup%
}}\right.$}%
\begingroup \smaller\smaller\smaller\begin{tabular}{@{}c@{}}%
2\\1\\0
\end{tabular}\endgroup%
\kern3pt%
\begingroup \smaller\smaller\smaller\begin{tabular}{@{}c@{}}%
24\\8\\6
\end{tabular}\endgroup%
\kern3pt%
\begingroup \smaller\smaller\smaller\begin{tabular}{@{}c@{}}%
24\\6\\8
\end{tabular}\endgroup%
\kern3pt%
\begingroup \smaller\smaller\smaller\begin{tabular}{@{}c@{}}%
12\\1\\5
\end{tabular}\endgroup%
\kern3pt%
\begingroup \smaller\smaller\smaller\begin{tabular}{@{}c@{}}%
12\\-1\\5
\end{tabular}\endgroup%
\kern3pt%
\begingroup \smaller\smaller\smaller\begin{tabular}{@{}c@{}}%
3\\-1\\1
\end{tabular}\endgroup%
\kern3pt%
\begingroup \smaller\smaller\smaller\begin{tabular}{@{}c@{}}%
12\\-5\\1
\end{tabular}\endgroup%
\kern3pt%
\begingroup \smaller\smaller\smaller\begin{tabular}{@{}c@{}}%
12\\-5\\-1
\end{tabular}\endgroup%
\kern3pt%
\begingroup \smaller\smaller\smaller\begin{tabular}{@{}c@{}}%
3\\-1\\-1
\end{tabular}\endgroup%
\kern3pt%
\begingroup \smaller\smaller\smaller\begin{tabular}{@{}c@{}}%
12\\-1\\-5
\end{tabular}\endgroup%
\kern3pt%
\begingroup \smaller\smaller\smaller\begin{tabular}{@{}c@{}}%
12\\1\\-5
\end{tabular}\endgroup%
\kern3pt%
\begingroup \smaller\smaller\smaller\begin{tabular}{@{}c@{}}%
3\\1\\-1
\end{tabular}\endgroup%
{$\left.\llap{\phantom{%
\begingroup \smaller\smaller\smaller\begin{tabular}{@{}c@{}}%
0\\0\\0
\end{tabular}\endgroup%
}}\!\right]$}%
}%
\ifdim\wd\matricesbox>\halfwidth\myboxwidth=\hsize\else\myboxwidth=\halfwidth\fi
\vbox{%
\ifdim\myboxwidth=\hsize
\setbox\onelinebox=\hbox{%
\vbox{\hbox{%
$\Pi_{12,30}$ spans $L_{123.8}$%
}\hbox{%
$224222222222$%
}%
}%
\hfill\copy\matricesbox
}%
\ifdim\wd\onelinebox>\myboxwidth
\hbox to \myboxwidth{%
$\Pi_{12,30}$ spans $L_{123.8}$%
\hfil
$224222222222$%
}%
\box\matricesbox
\else
\hbox to \myboxwidth{%
\unhbox\onelinebox
}%
\fi
\else
\hbox to \myboxwidth{%
$\Pi_{12,30}$ spans $L_{123.8}$%
\hfil}%
\hbox to \myboxwidth{%
$224222222222$%
\hfil}%
\box\matricesbox
\fi
}%
\hfill\discretionary{}{}{}%
\setbox\matricesbox=\hbox{%
{$\left[\!\llap{\phantom{%
\begingroup \smaller\smaller\smaller
\begin{tabular}{@{}c@{}}%
\phantom{0}\\\phantom{0}\\\phantom{0}
\end{tabular}\endgroup%
}}\right.$}%
\begingroup \smaller\smaller\smaller\begin{tabular}{@{}c@{}}%
-1\\\phantom{0}\\\phantom{0}
\end{tabular}\endgroup%
\kern3pt%
\begingroup \smaller\smaller\smaller\begin{tabular}{@{}c@{}}%
\phantom{0}\\6\\\phantom{0}
\end{tabular}\endgroup%
\kern3pt%
\begingroup \smaller\smaller\smaller\begin{tabular}{@{}c@{}}%
\phantom{0}\\\phantom{0}\\6
\end{tabular}\endgroup%
{$\left.\llap{\phantom{%
\begingroup \smaller\smaller\smaller\begin{tabular}{@{}c@{}}%
\phantom{0}\\\phantom{0}\\\phantom{0}
\end{tabular}\endgroup%
}}\!\right]$}%
{$\left[\!\llap{\phantom{%
\begingroup \smaller\smaller\smaller\begin{tabular}{@{}c@{}}%
0\\0\\0
\end{tabular}\endgroup%
}}\right.$}%
\begingroup \smaller\smaller\smaller\begin{tabular}{@{}c@{}}%
2\\1\\0
\end{tabular}\endgroup%
\kern3pt%
\begingroup \smaller\smaller\smaller\begin{tabular}{@{}c@{}}%
24\\8\\6
\end{tabular}\endgroup%
\kern3pt%
\begingroup \smaller\smaller\smaller\begin{tabular}{@{}c@{}}%
24\\6\\8
\end{tabular}\endgroup%
\kern3pt%
\begingroup \smaller\smaller\smaller\begin{tabular}{@{}c@{}}%
12\\1\\5
\end{tabular}\endgroup%
\kern3pt%
\begingroup \smaller\smaller\smaller\begin{tabular}{@{}c@{}}%
12\\-1\\5
\end{tabular}\endgroup%
\kern3pt%
\begingroup \smaller\smaller\smaller\begin{tabular}{@{}c@{}}%
24\\-6\\8
\end{tabular}\endgroup%
\kern3pt%
\begingroup \smaller\smaller\smaller\begin{tabular}{@{}c@{}}%
24\\-8\\6
\end{tabular}\endgroup%
\kern3pt%
\begingroup \smaller\smaller\smaller\begin{tabular}{@{}c@{}}%
2\\-1\\0
\end{tabular}\endgroup%
\kern3pt%
\begingroup \smaller\smaller\smaller\begin{tabular}{@{}c@{}}%
3\\-1\\-1
\end{tabular}\endgroup%
\kern3pt%
\begingroup \smaller\smaller\smaller\begin{tabular}{@{}c@{}}%
2\\0\\-1
\end{tabular}\endgroup%
\kern3pt%
\begingroup \smaller\smaller\smaller\begin{tabular}{@{}c@{}}%
24\\6\\-8
\end{tabular}\endgroup%
\kern3pt%
\begingroup \smaller\smaller\smaller\begin{tabular}{@{}c@{}}%
24\\8\\-6
\end{tabular}\endgroup%
{$\left.\llap{\phantom{%
\begingroup \smaller\smaller\smaller\begin{tabular}{@{}c@{}}%
0\\0\\0
\end{tabular}\endgroup%
}}\!\right]$}%
}%
\ifdim\wd\matricesbox>\halfwidth\myboxwidth=\hsize\else\myboxwidth=\halfwidth\fi
\vbox{%
\ifdim\myboxwidth=\hsize
\setbox\onelinebox=\hbox{%
\vbox{\hbox{%
$\Pi_{12,31}$ spans $L_{123.8}$%
}\hbox{%
$224242222222$%
}%
}%
\hfill\copy\matricesbox
}%
\ifdim\wd\onelinebox>\myboxwidth
\hbox to \myboxwidth{%
$\Pi_{12,31}$ spans $L_{123.8}$%
\hfil
$224242222222$%
}%
\box\matricesbox
\else
\hbox to \myboxwidth{%
\unhbox\onelinebox
}%
\fi
\else
\hbox to \myboxwidth{%
$\Pi_{12,31}$ spans $L_{123.8}$%
\hfil}%
\hbox to \myboxwidth{%
$224242222222$%
\hfil}%
\box\matricesbox
\fi
}%
\hfill\discretionary{}{}{}%
\setbox\matricesbox=\hbox{%
{$\left[\!\llap{\phantom{%
\begingroup \smaller\smaller\smaller
\begin{tabular}{@{}c@{}}%
\phantom{0}\\\phantom{0}\\\phantom{0}
\end{tabular}\endgroup%
}}\right.$}%
\begingroup \smaller\smaller\smaller\begin{tabular}{@{}c@{}}%
-1\\\phantom{0}\\\phantom{0}
\end{tabular}\endgroup%
\kern3pt%
\begingroup \smaller\smaller\smaller\begin{tabular}{@{}c@{}}%
\phantom{0}\\6\\\phantom{0}
\end{tabular}\endgroup%
\kern3pt%
\begingroup \smaller\smaller\smaller\begin{tabular}{@{}c@{}}%
\phantom{0}\\\phantom{0}\\6
\end{tabular}\endgroup%
{$\left.\llap{\phantom{%
\begingroup \smaller\smaller\smaller\begin{tabular}{@{}c@{}}%
\phantom{0}\\\phantom{0}\\\phantom{0}
\end{tabular}\endgroup%
}}\!\right]$}%
{$\left[\!\llap{\phantom{%
\begingroup \smaller\smaller\smaller\begin{tabular}{@{}c@{}}%
0\\0\\0
\end{tabular}\endgroup%
}}\right.$}%
\begingroup \smaller\smaller\smaller\begin{tabular}{@{}c@{}}%
2\\1\\0
\end{tabular}\endgroup%
\kern3pt%
\begingroup \smaller\smaller\smaller\begin{tabular}{@{}c@{}}%
24\\8\\6
\end{tabular}\endgroup%
\kern3pt%
\begingroup \smaller\smaller\smaller\begin{tabular}{@{}c@{}}%
24\\6\\8
\end{tabular}\endgroup%
\kern3pt%
\begingroup \smaller\smaller\smaller\begin{tabular}{@{}c@{}}%
12\\1\\5
\end{tabular}\endgroup%
\kern3pt%
\begingroup \smaller\smaller\smaller\begin{tabular}{@{}c@{}}%
12\\-1\\5
\end{tabular}\endgroup%
\kern3pt%
\begingroup \smaller\smaller\smaller\begin{tabular}{@{}c@{}}%
24\\-6\\8
\end{tabular}\endgroup%
\kern3pt%
\begingroup \smaller\smaller\smaller\begin{tabular}{@{}c@{}}%
24\\-8\\6
\end{tabular}\endgroup%
\kern3pt%
\begingroup \smaller\smaller\smaller\begin{tabular}{@{}c@{}}%
12\\-5\\1
\end{tabular}\endgroup%
\kern3pt%
\begingroup \smaller\smaller\smaller\begin{tabular}{@{}c@{}}%
12\\-5\\-1
\end{tabular}\endgroup%
\kern3pt%
\begingroup \smaller\smaller\smaller\begin{tabular}{@{}c@{}}%
3\\-1\\-1
\end{tabular}\endgroup%
\kern3pt%
\begingroup \smaller\smaller\smaller\begin{tabular}{@{}c@{}}%
2\\0\\-1
\end{tabular}\endgroup%
\kern3pt%
\begingroup \smaller\smaller\smaller\begin{tabular}{@{}c@{}}%
3\\1\\-1
\end{tabular}\endgroup%
{$\left.\llap{\phantom{%
\begingroup \smaller\smaller\smaller\begin{tabular}{@{}c@{}}%
0\\0\\0
\end{tabular}\endgroup%
}}\!\right]$}%
}%
\ifdim\wd\matricesbox>\halfwidth\myboxwidth=\hsize\else\myboxwidth=\halfwidth\fi
\vbox{%
\ifdim\myboxwidth=\hsize
\setbox\onelinebox=\hbox{%
\vbox{\hbox{%
$\Pi_{12,32}$ spans $L_{123.8}$%
}\hbox{%
$224242422222$%
}%
}%
\hfill\copy\matricesbox
}%
\ifdim\wd\onelinebox>\myboxwidth
\hbox to \myboxwidth{%
$\Pi_{12,32}$ spans $L_{123.8}$%
\hfil
$224242422222$%
}%
\box\matricesbox
\else
\hbox to \myboxwidth{%
\unhbox\onelinebox
}%
\fi
\else
\hbox to \myboxwidth{%
$\Pi_{12,32}$ spans $L_{123.8}$%
\hfil}%
\hbox to \myboxwidth{%
$224242422222$%
\hfil}%
\box\matricesbox
\fi
}%
\hfill\discretionary{}{}{}%
\setbox\matricesbox=\hbox{%
{$\left[\!\llap{\phantom{%
\begingroup \smaller\smaller\smaller
\begin{tabular}{@{}c@{}}%
\phantom{0}\\\phantom{0}\\\phantom{0}
\end{tabular}\endgroup%
}}\right.$}%
\begingroup \smaller\smaller\smaller\begin{tabular}{@{}c@{}}%
-1\\\phantom{0}\\\phantom{0}
\end{tabular}\endgroup%
\kern3pt%
\begingroup \smaller\smaller\smaller\begin{tabular}{@{}c@{}}%
\phantom{0}\\6\\\phantom{0}
\end{tabular}\endgroup%
\kern3pt%
\begingroup \smaller\smaller\smaller\begin{tabular}{@{}c@{}}%
\phantom{0}\\\phantom{0}\\6
\end{tabular}\endgroup%
{$\left.\llap{\phantom{%
\begingroup \smaller\smaller\smaller\begin{tabular}{@{}c@{}}%
\phantom{0}\\\phantom{0}\\\phantom{0}
\end{tabular}\endgroup%
}}\!\right]$}%
{$\left[\!\llap{\phantom{%
\begingroup \smaller\smaller\smaller\begin{tabular}{@{}c@{}}%
0\\0\\0
\end{tabular}\endgroup%
}}\right.$}%
\begingroup \smaller\smaller\smaller\begin{tabular}{@{}c@{}}%
3\\-1\\-1
\end{tabular}\endgroup%
\kern3pt%
\begingroup \smaller\smaller\smaller\begin{tabular}{@{}c@{}}%
12\\-1\\-5
\end{tabular}\endgroup%
\kern3pt%
\begingroup \smaller\smaller\smaller\begin{tabular}{@{}c@{}}%
12\\1\\-5
\end{tabular}\endgroup%
\kern3pt%
\begingroup \smaller\smaller\smaller\begin{tabular}{@{}c@{}}%
24\\6\\-8
\end{tabular}\endgroup%
\kern3pt%
\begingroup \smaller\smaller\smaller\begin{tabular}{@{}c@{}}%
24\\8\\-6
\end{tabular}\endgroup%
\kern3pt%
\begingroup \smaller\smaller\smaller\begin{tabular}{@{}c@{}}%
12\\5\\-1
\end{tabular}\endgroup%
\kern3pt%
\begingroup \smaller\smaller\smaller\begin{tabular}{@{}c@{}}%
12\\5\\1
\end{tabular}\endgroup%
\kern3pt%
\begingroup \smaller\smaller\smaller\begin{tabular}{@{}c@{}}%
3\\1\\1
\end{tabular}\endgroup%
\kern3pt%
\begingroup \smaller\smaller\smaller\begin{tabular}{@{}c@{}}%
2\\0\\1
\end{tabular}\endgroup%
\kern3pt%
\begingroup \smaller\smaller\smaller\begin{tabular}{@{}c@{}}%
3\\-1\\1
\end{tabular}\endgroup%
\kern3pt%
\begingroup \smaller\smaller\smaller\begin{tabular}{@{}c@{}}%
12\\-5\\1
\end{tabular}\endgroup%
\kern3pt%
\begingroup \smaller\smaller\smaller\begin{tabular}{@{}c@{}}%
12\\-5\\-1
\end{tabular}\endgroup%
{$\left.\llap{\phantom{%
\begingroup \smaller\smaller\smaller\begin{tabular}{@{}c@{}}%
0\\0\\0
\end{tabular}\endgroup%
}}\!\right]$}%
}%
\ifdim\wd\matricesbox>\halfwidth\myboxwidth=\hsize\else\myboxwidth=\halfwidth\fi
\vbox{%
\ifdim\myboxwidth=\hsize
\setbox\onelinebox=\hbox{%
\vbox{\hbox{%
$\Pi_{12,33}$ spans $L_{123.8}$%
}\hbox{%
$224242222222$%
}%
}%
\hfill\copy\matricesbox
}%
\ifdim\wd\onelinebox>\myboxwidth
\hbox to \myboxwidth{%
$\Pi_{12,33}$ spans $L_{123.8}$%
\hfil
$224242222222$%
}%
\box\matricesbox
\else
\hbox to \myboxwidth{%
\unhbox\onelinebox
}%
\fi
\else
\hbox to \myboxwidth{%
$\Pi_{12,33}$ spans $L_{123.8}$%
\hfil}%
\hbox to \myboxwidth{%
$224242222222$%
\hfil}%
\box\matricesbox
\fi
}%
\hfill\discretionary{}{}{}%

\vskip2pt\hrule\vskip2pt

\leavevmode\setbox\matricesbox=\hbox{%
{$\left[\!\llap{\phantom{%
\begingroup \smaller\smaller\smaller\begin{tabular}{@{}c@{}}%
\phantom{0}\\\phantom{0}\\\phantom{0}
\end{tabular}\endgroup%
}}\right.$}%
\begingroup \smaller\smaller\smaller\begin{tabular}{@{}c@{}}%
-1\\\phantom{0}\\\phantom{0}
\end{tabular}\endgroup%
\kern3pt%
\begingroup \smaller\smaller\smaller\begin{tabular}{@{}c@{}}%
\phantom{0}\\5\\\phantom{0}
\end{tabular}\endgroup%
\kern3pt%
\begingroup \smaller\smaller\smaller\begin{tabular}{@{}c@{}}%
\phantom{0}\\\phantom{0}\\15
\end{tabular}\endgroup%
{$\left.\llap{\phantom{%
\begingroup \smaller\smaller\smaller\begin{tabular}{@{}c@{}}%
\phantom{0}\\\phantom{0}\\\phantom{0}
\end{tabular}\endgroup%
}}\!\right]$}%
{$\left[\!\llap{\phantom{%
\begingroup \smaller\smaller\smaller\begin{tabular}{@{}c@{}}%
0\\0\\0
\end{tabular}\endgroup%
}}\right.$}%
\begingroup \smaller\smaller\smaller\begin{tabular}{@{}c@{}}%
4\\-2\\0
\end{tabular}\endgroup%
\kern3pt%
\begingroup \smaller\smaller\smaller\begin{tabular}{@{}c@{}}%
15\\-6\\-2
\end{tabular}\endgroup%
\kern3pt%
\begingroup \smaller\smaller\smaller\begin{tabular}{@{}c@{}}%
4\\-1\\-1
\end{tabular}\endgroup%
\kern3pt%
\begingroup \smaller\smaller\smaller\begin{tabular}{@{}c@{}}%
15\\0\\-4
\end{tabular}\endgroup%
\kern3pt%
\begingroup \smaller\smaller\smaller\begin{tabular}{@{}c@{}}%
4\\1\\-1
\end{tabular}\endgroup%
\kern3pt%
\begingroup \smaller\smaller\smaller\begin{tabular}{@{}c@{}}%
15\\6\\-2
\end{tabular}\endgroup%
\kern3pt%
\begingroup \smaller\smaller\smaller\begin{tabular}{@{}c@{}}%
20\\9\\-1
\end{tabular}\endgroup%
{$\left.\llap{\phantom{%
\begingroup \smaller\smaller\smaller\begin{tabular}{@{}c@{}}%
0\\0\\0
\end{tabular}\endgroup%
}}\!\right]$}%
}%
\ifdim\wd\matricesbox>\halfwidth\myboxwidth=\hsize\else\myboxwidth=\halfwidth\fi
\vbox{%
\ifdim\myboxwidth=\hsize
\setbox\onelinebox=\hbox{%
\vbox{\hbox{%
$\Pi_{13,1}$ spans $L_{31.7}$%
}\hbox{%
$2222|222222\slashthree22\rtimes D_{2}$%
}%
}%
\hfill\copy\matricesbox
}%
\ifdim\wd\onelinebox>\myboxwidth
\hbox to \myboxwidth{%
$\Pi_{13,1}$ spans $L_{31.7}$%
\hfil
$2222|222222\slashthree22\rtimes D_{2}$%
}%
\box\matricesbox
\else
\hbox to \myboxwidth{%
\unhbox\onelinebox
}%
\fi
\else
\hbox to \myboxwidth{%
$\Pi_{13,1}$ spans $L_{31.7}$%
\hfil}%
\hbox to \myboxwidth{%
$2222|222222\slashthree22\rtimes D_{2}$%
\hfil}%
\box\matricesbox
\fi
}%
\hfill\discretionary{}{}{}%
\setbox\matricesbox=\hbox{%
{$\left[\!\llap{\phantom{%
\begingroup \smaller\smaller\smaller\begin{tabular}{@{}c@{}}%
\phantom{0}\\\phantom{0}\\\phantom{0}
\end{tabular}\endgroup%
}}\right.$}%
\begingroup \smaller\smaller\smaller\begin{tabular}{@{}c@{}}%
-1\\\phantom{0}\\\phantom{0}
\end{tabular}\endgroup%
\kern3pt%
\begingroup \smaller\smaller\smaller\begin{tabular}{@{}c@{}}%
\phantom{0}\\45/2\\\phantom{0}
\end{tabular}\endgroup%
\kern3pt%
\begingroup \smaller\smaller\smaller\begin{tabular}{@{}c@{}}%
\phantom{0}\\\phantom{0}\\15/2
\end{tabular}\endgroup%
{$\left.\llap{\phantom{%
\begingroup \smaller\smaller\smaller\begin{tabular}{@{}c@{}}%
\phantom{0}\\\phantom{0}\\\phantom{0}
\end{tabular}\endgroup%
}}\!\right]$}%
{$\left[\!\llap{\phantom{%
\begingroup \smaller\smaller\smaller\begin{tabular}{@{}c@{}}%
0\\0\\0
\end{tabular}\endgroup%
}}\right.$}%
\begingroup \smaller\smaller\smaller\begin{tabular}{@{}c@{}}%
9\\-2\\0
\end{tabular}\endgroup%
\kern3pt%
\begingroup \smaller\smaller\smaller\begin{tabular}{@{}c@{}}%
5\\-1\\1
\end{tabular}\endgroup%
\kern3pt%
\begingroup \smaller\smaller\smaller\begin{tabular}{@{}c@{}}%
9\\-1\\3
\end{tabular}\endgroup%
\kern3pt%
\begingroup \smaller\smaller\smaller\begin{tabular}{@{}c@{}}%
5\\0\\2
\end{tabular}\endgroup%
\kern3pt%
\begingroup \smaller\smaller\smaller\begin{tabular}{@{}c@{}}%
9\\1\\3
\end{tabular}\endgroup%
\kern3pt%
\begingroup \smaller\smaller\smaller\begin{tabular}{@{}c@{}}%
5\\1\\1
\end{tabular}\endgroup%
\kern3pt%
\begingroup \smaller\smaller\smaller\begin{tabular}{@{}c@{}}%
90\\19\\3
\end{tabular}\endgroup%
{$\left.\llap{\phantom{%
\begingroup \smaller\smaller\smaller\begin{tabular}{@{}c@{}}%
0\\0\\0
\end{tabular}\endgroup%
}}\!\right]$}%
}%
\ifdim\wd\matricesbox>\halfwidth\myboxwidth=\hsize\else\myboxwidth=\halfwidth\fi
\vbox{%
\ifdim\myboxwidth=\hsize
\setbox\onelinebox=\hbox{%
\vbox{\hbox{%
$\Pi_{13,2}$ spans $L_{16.13}$%
}\hbox{%
$22222|222222\slashthree2\rtimes D_{2}$%
}%
}%
\hfill\copy\matricesbox
}%
\ifdim\wd\onelinebox>\myboxwidth
\hbox to \myboxwidth{%
$\Pi_{13,2}$ spans $L_{16.13}$%
\hfil
$22222|222222\slashthree2\rtimes D_{2}$%
}%
\box\matricesbox
\else
\hbox to \myboxwidth{%
\unhbox\onelinebox
}%
\fi
\else
\hbox to \myboxwidth{%
$\Pi_{13,2}$ spans $L_{16.13}$%
\hfil}%
\hbox to \myboxwidth{%
$22222|222222\slashthree2\rtimes D_{2}$%
\hfil}%
\box\matricesbox
\fi
}%
\hfill\discretionary{}{}{}%
\setbox\matricesbox=\hbox{%
{$\left[\!\llap{\phantom{%
\begingroup \smaller\smaller\smaller\begin{tabular}{@{}c@{}}%
\phantom{0}\\\phantom{0}\\\phantom{0}
\end{tabular}\endgroup%
}}\right.$}%
\begingroup \smaller\smaller\smaller\begin{tabular}{@{}c@{}}%
-1\\\phantom{0}\\\phantom{0}
\end{tabular}\endgroup%
\kern3pt%
\begingroup \smaller\smaller\smaller\begin{tabular}{@{}c@{}}%
\phantom{0}\\15/2\\\phantom{0}
\end{tabular}\endgroup%
\kern3pt%
\begingroup \smaller\smaller\smaller\begin{tabular}{@{}c@{}}%
\phantom{0}\\\phantom{0}\\45/2
\end{tabular}\endgroup%
{$\left.\llap{\phantom{%
\begingroup \smaller\smaller\smaller\begin{tabular}{@{}c@{}}%
\phantom{0}\\\phantom{0}\\\phantom{0}
\end{tabular}\endgroup%
}}\!\right]$}%
{$\left[\!\llap{\phantom{%
\begingroup \smaller\smaller\smaller\begin{tabular}{@{}c@{}}%
0\\0\\0
\end{tabular}\endgroup%
}}\right.$}%
\begingroup \smaller\smaller\smaller\begin{tabular}{@{}c@{}}%
5\\-2\\0
\end{tabular}\endgroup%
\kern3pt%
\begingroup \smaller\smaller\smaller\begin{tabular}{@{}c@{}}%
9\\-3\\1
\end{tabular}\endgroup%
\kern3pt%
\begingroup \smaller\smaller\smaller\begin{tabular}{@{}c@{}}%
5\\-1\\1
\end{tabular}\endgroup%
\kern3pt%
\begingroup \smaller\smaller\smaller\begin{tabular}{@{}c@{}}%
9\\0\\2
\end{tabular}\endgroup%
\kern3pt%
\begingroup \smaller\smaller\smaller\begin{tabular}{@{}c@{}}%
5\\1\\1
\end{tabular}\endgroup%
\kern3pt%
\begingroup \smaller\smaller\smaller\begin{tabular}{@{}c@{}}%
9\\3\\1
\end{tabular}\endgroup%
\kern3pt%
\begingroup \smaller\smaller\smaller\begin{tabular}{@{}c@{}}%
30\\11\\1
\end{tabular}\endgroup%
{$\left.\llap{\phantom{%
\begingroup \smaller\smaller\smaller\begin{tabular}{@{}c@{}}%
0\\0\\0
\end{tabular}\endgroup%
}}\!\right]$}%
}%
\ifdim\wd\matricesbox>\halfwidth\myboxwidth=\hsize\else\myboxwidth=\halfwidth\fi
\vbox{%
\ifdim\myboxwidth=\hsize
\setbox\onelinebox=\hbox{%
\vbox{\hbox{%
$\Pi_{13,3}$ spans $L_{16.13}$%
}\hbox{%
$2222|222222\slashthree22\rtimes D_{2}$%
}%
}%
\hfill\copy\matricesbox
}%
\ifdim\wd\onelinebox>\myboxwidth
\hbox to \myboxwidth{%
$\Pi_{13,3}$ spans $L_{16.13}$%
\hfil
$2222|222222\slashthree22\rtimes D_{2}$%
}%
\box\matricesbox
\else
\hbox to \myboxwidth{%
\unhbox\onelinebox
}%
\fi
\else
\hbox to \myboxwidth{%
$\Pi_{13,3}$ spans $L_{16.13}$%
\hfil}%
\hbox to \myboxwidth{%
$2222|222222\slashthree22\rtimes D_{2}$%
\hfil}%
\box\matricesbox
\fi
}%
\hfill\discretionary{}{}{}%
\setbox\matricesbox=\hbox{%
{$\left[\!\llap{\phantom{%
\begingroup \smaller\smaller\smaller\begin{tabular}{@{}c@{}}%
\phantom{0}\\\phantom{0}\\\phantom{0}
\end{tabular}\endgroup%
}}\right.$}%
\begingroup \smaller\smaller\smaller\begin{tabular}{@{}c@{}}%
-1\\\phantom{0}\\\phantom{0}
\end{tabular}\endgroup%
\kern3pt%
\begingroup \smaller\smaller\smaller\begin{tabular}{@{}c@{}}%
\phantom{0}\\21/2\\\phantom{0}
\end{tabular}\endgroup%
\kern3pt%
\begingroup \smaller\smaller\smaller\begin{tabular}{@{}c@{}}%
\phantom{0}\\\phantom{0}\\7/2
\end{tabular}\endgroup%
{$\left.\llap{\phantom{%
\begingroup \smaller\smaller\smaller\begin{tabular}{@{}c@{}}%
\phantom{0}\\\phantom{0}\\\phantom{0}
\end{tabular}\endgroup%
}}\!\right]$}%
{$\left[\!\llap{\phantom{%
\begingroup \smaller\smaller\smaller\begin{tabular}{@{}c@{}}%
0\\0\\0
\end{tabular}\endgroup%
}}\right.$}%
\begingroup \smaller\smaller\smaller\begin{tabular}{@{}c@{}}%
6\\-2\\0
\end{tabular}\endgroup%
\kern3pt%
\begingroup \smaller\smaller\smaller\begin{tabular}{@{}c@{}}%
7\\-2\\-2
\end{tabular}\endgroup%
\kern3pt%
\begingroup \smaller\smaller\smaller\begin{tabular}{@{}c@{}}%
6\\-1\\-3
\end{tabular}\endgroup%
\kern3pt%
\begingroup \smaller\smaller\smaller\begin{tabular}{@{}c@{}}%
7\\0\\-4
\end{tabular}\endgroup%
\kern3pt%
\begingroup \smaller\smaller\smaller\begin{tabular}{@{}c@{}}%
6\\1\\-3
\end{tabular}\endgroup%
\kern3pt%
\begingroup \smaller\smaller\smaller\begin{tabular}{@{}c@{}}%
7\\2\\-2
\end{tabular}\endgroup%
\kern3pt%
\begingroup \smaller\smaller\smaller\begin{tabular}{@{}c@{}}%
42\\13\\-3
\end{tabular}\endgroup%
{$\left.\llap{\phantom{%
\begingroup \smaller\smaller\smaller\begin{tabular}{@{}c@{}}%
0\\0\\0
\end{tabular}\endgroup%
}}\!\right]$}%
}%
\ifdim\wd\matricesbox>\halfwidth\myboxwidth=\hsize\else\myboxwidth=\halfwidth\fi
\vbox{%
\ifdim\myboxwidth=\hsize
\setbox\onelinebox=\hbox{%
\vbox{\hbox{%
$\Pi_{13,4}$ spans $L_{22.4}$%
}\hbox{%
$2222|222222\slashthree22\rtimes D_{2}$%
}%
}%
\hfill\copy\matricesbox
}%
\ifdim\wd\onelinebox>\myboxwidth
\hbox to \myboxwidth{%
$\Pi_{13,4}$ spans $L_{22.4}$%
\hfil
$2222|222222\slashthree22\rtimes D_{2}$%
}%
\box\matricesbox
\else
\hbox to \myboxwidth{%
\unhbox\onelinebox
}%
\fi
\else
\hbox to \myboxwidth{%
$\Pi_{13,4}$ spans $L_{22.4}$%
\hfil}%
\hbox to \myboxwidth{%
$2222|222222\slashthree22\rtimes D_{2}$%
\hfil}%
\box\matricesbox
\fi
}%
\hfill\discretionary{}{}{}%
\setbox\matricesbox=\hbox{%
{$\left[\!\llap{\phantom{%
\begingroup \smaller\smaller\smaller\begin{tabular}{@{}c@{}}%
\phantom{0}\\\phantom{0}\\\phantom{0}
\end{tabular}\endgroup%
}}\right.$}%
\begingroup \smaller\smaller\smaller\begin{tabular}{@{}c@{}}%
-1\\\phantom{0}\\\phantom{0}
\end{tabular}\endgroup%
\kern3pt%
\begingroup \smaller\smaller\smaller\begin{tabular}{@{}c@{}}%
\phantom{0}\\6\\\phantom{0}
\end{tabular}\endgroup%
\kern3pt%
\begingroup \smaller\smaller\smaller\begin{tabular}{@{}c@{}}%
\phantom{0}\\\phantom{0}\\6
\end{tabular}\endgroup%
{$\left.\llap{\phantom{%
\begingroup \smaller\smaller\smaller\begin{tabular}{@{}c@{}}%
\phantom{0}\\\phantom{0}\\\phantom{0}
\end{tabular}\endgroup%
}}\!\right]$}%
{$\left[\!\llap{\phantom{%
\begingroup \smaller\smaller\smaller\begin{tabular}{@{}c@{}}%
0\\0\\0
\end{tabular}\endgroup%
}}\right.$}%
\begingroup \smaller\smaller\smaller\begin{tabular}{@{}c@{}}%
2\\-1\\0
\end{tabular}\endgroup%
\kern3pt%
\begingroup \smaller\smaller\smaller\begin{tabular}{@{}c@{}}%
24\\-8\\-6
\end{tabular}\endgroup%
\kern3pt%
\begingroup \smaller\smaller\smaller\begin{tabular}{@{}c@{}}%
24\\-6\\-8
\end{tabular}\endgroup%
\kern3pt%
\begingroup \smaller\smaller\smaller\begin{tabular}{@{}c@{}}%
12\\-1\\-5
\end{tabular}\endgroup%
\kern3pt%
\begingroup \smaller\smaller\smaller\begin{tabular}{@{}c@{}}%
12\\1\\-5
\end{tabular}\endgroup%
\kern3pt%
\begingroup \smaller\smaller\smaller\begin{tabular}{@{}c@{}}%
3\\1\\-1
\end{tabular}\endgroup%
\kern3pt%
\begingroup \smaller\smaller\smaller\begin{tabular}{@{}c@{}}%
12\\5\\-1
\end{tabular}\endgroup%
{$\left.\llap{\phantom{%
\begingroup \smaller\smaller\smaller\begin{tabular}{@{}c@{}}%
0\\0\\0
\end{tabular}\endgroup%
}}\!\right]$}%
}%
\ifdim\wd\matricesbox>\halfwidth\myboxwidth=\hsize\else\myboxwidth=\halfwidth\fi
\vbox{%
\ifdim\myboxwidth=\hsize
\setbox\onelinebox=\hbox{%
\vbox{\hbox{%
$\Pi_{13,5}$ spans $L_{123.8}$%
}\hbox{%
$|224222\slashtwo222422\rtimes D_{2}$%
}%
}%
\hfill\copy\matricesbox
}%
\ifdim\wd\onelinebox>\myboxwidth
\hbox to \myboxwidth{%
$\Pi_{13,5}$ spans $L_{123.8}$%
\hfil
$|224222\slashtwo222422\rtimes D_{2}$%
}%
\box\matricesbox
\else
\hbox to \myboxwidth{%
\unhbox\onelinebox
}%
\fi
\else
\hbox to \myboxwidth{%
$\Pi_{13,5}$ spans $L_{123.8}$%
\hfil}%
\hbox to \myboxwidth{%
$|224222\slashtwo222422\rtimes D_{2}$%
\hfil}%
\box\matricesbox
\fi
}%
\hfill\discretionary{}{}{}%
\setbox\matricesbox=\hbox{%
{$\left[\!\llap{\phantom{%
\begingroup \smaller\smaller\smaller\begin{tabular}{@{}c@{}}%
\phantom{0}\\\phantom{0}\\\phantom{0}
\end{tabular}\endgroup%
}}\right.$}%
\begingroup \smaller\smaller\smaller\begin{tabular}{@{}c@{}}%
-1\\\phantom{0}\\\phantom{0}
\end{tabular}\endgroup%
\kern3pt%
\begingroup \smaller\smaller\smaller\begin{tabular}{@{}c@{}}%
\phantom{0}\\3\\\phantom{0}
\end{tabular}\endgroup%
\kern3pt%
\begingroup \smaller\smaller\smaller\begin{tabular}{@{}c@{}}%
\phantom{0}\\\phantom{0}\\3
\end{tabular}\endgroup%
{$\left.\llap{\phantom{%
\begingroup \smaller\smaller\smaller\begin{tabular}{@{}c@{}}%
\phantom{0}\\\phantom{0}\\\phantom{0}
\end{tabular}\endgroup%
}}\!\right]$}%
{$\left[\!\llap{\phantom{%
\begingroup \smaller\smaller\smaller\begin{tabular}{@{}c@{}}%
0\\0\\0
\end{tabular}\endgroup%
}}\right.$}%
\begingroup \smaller\smaller\smaller\begin{tabular}{@{}c@{}}%
3\\-2\\0
\end{tabular}\endgroup%
\kern3pt%
\begingroup \smaller\smaller\smaller\begin{tabular}{@{}c@{}}%
12\\-6\\4
\end{tabular}\endgroup%
\kern3pt%
\begingroup \smaller\smaller\smaller\begin{tabular}{@{}c@{}}%
12\\-4\\6
\end{tabular}\endgroup%
\kern3pt%
\begingroup \smaller\smaller\smaller\begin{tabular}{@{}c@{}}%
24\\-2\\14
\end{tabular}\endgroup%
\kern3pt%
\begingroup \smaller\smaller\smaller\begin{tabular}{@{}c@{}}%
24\\2\\14
\end{tabular}\endgroup%
\kern3pt%
\begingroup \smaller\smaller\smaller\begin{tabular}{@{}c@{}}%
2\\1\\1
\end{tabular}\endgroup%
\kern3pt%
\begingroup \smaller\smaller\smaller\begin{tabular}{@{}c@{}}%
24\\14\\2
\end{tabular}\endgroup%
{$\left.\llap{\phantom{%
\begingroup \smaller\smaller\smaller\begin{tabular}{@{}c@{}}%
0\\0\\0
\end{tabular}\endgroup%
}}\!\right]$}%
}%
\ifdim\wd\matricesbox>\halfwidth\myboxwidth=\hsize\else\myboxwidth=\halfwidth\fi
\vbox{%
\ifdim\myboxwidth=\hsize
\setbox\onelinebox=\hbox{%
\vbox{\hbox{%
$\Pi_{13,6}$ spans $L_{123.8}$%
}\hbox{%
$22422|224222\slashtwo2\rtimes D_{2}$%
}%
}%
\hfill\copy\matricesbox
}%
\ifdim\wd\onelinebox>\myboxwidth
\hbox to \myboxwidth{%
$\Pi_{13,6}$ spans $L_{123.8}$%
\hfil
$22422|224222\slashtwo2\rtimes D_{2}$%
}%
\box\matricesbox
\else
\hbox to \myboxwidth{%
\unhbox\onelinebox
}%
\fi
\else
\hbox to \myboxwidth{%
$\Pi_{13,6}$ spans $L_{123.8}$%
\hfil}%
\hbox to \myboxwidth{%
$22422|224222\slashtwo2\rtimes D_{2}$%
\hfil}%
\box\matricesbox
\fi
}%
\hfill\discretionary{}{}{}%
\setbox\matricesbox=\hbox{%
{$\left[\!\llap{\phantom{%
\begingroup \smaller\smaller\smaller\begin{tabular}{@{}c@{}}%
\phantom{0}\\\phantom{0}\\\phantom{0}
\end{tabular}\endgroup%
}}\right.$}%
\begingroup \smaller\smaller\smaller\begin{tabular}{@{}c@{}}%
-2\\\phantom{0}\\\phantom{0}
\end{tabular}\endgroup%
\kern3pt%
\begingroup \smaller\smaller\smaller\begin{tabular}{@{}c@{}}%
\phantom{0}\\3\\\phantom{0}
\end{tabular}\endgroup%
\kern3pt%
\begingroup \smaller\smaller\smaller\begin{tabular}{@{}c@{}}%
\phantom{0}\\\phantom{0}\\3
\end{tabular}\endgroup%
{$\left.\llap{\phantom{%
\begingroup \smaller\smaller\smaller\begin{tabular}{@{}c@{}}%
\phantom{0}\\\phantom{0}\\\phantom{0}
\end{tabular}\endgroup%
}}\!\right]$}%
{$\left[\!\llap{\phantom{%
\begingroup \smaller\smaller\smaller\begin{tabular}{@{}c@{}}%
0\\0\\0
\end{tabular}\endgroup%
}}\right.$}%
\begingroup \smaller\smaller\smaller\begin{tabular}{@{}c@{}}%
6\\5\\1
\end{tabular}\endgroup%
\kern3pt%
\begingroup \smaller\smaller\smaller\begin{tabular}{@{}c@{}}%
12\\8\\6
\end{tabular}\endgroup%
\kern3pt%
\begingroup \smaller\smaller\smaller\begin{tabular}{@{}c@{}}%
12\\6\\8
\end{tabular}\endgroup%
\kern3pt%
\begingroup \smaller\smaller\smaller\begin{tabular}{@{}c@{}}%
1\\0\\1
\end{tabular}\endgroup%
\kern3pt%
\begingroup \smaller\smaller\smaller\begin{tabular}{@{}c@{}}%
12\\-6\\8
\end{tabular}\endgroup%
\kern3pt%
\begingroup \smaller\smaller\smaller\begin{tabular}{@{}c@{}}%
12\\-8\\6
\end{tabular}\endgroup%
\kern3pt%
\begingroup \smaller\smaller\smaller\begin{tabular}{@{}c@{}}%
1\\-1\\0
\end{tabular}\endgroup%
{$\left.\llap{\phantom{%
\begingroup \smaller\smaller\smaller\begin{tabular}{@{}c@{}}%
0\\0\\0
\end{tabular}\endgroup%
}}\!\right]$}%
}%
\ifdim\wd\matricesbox>\halfwidth\myboxwidth=\hsize\else\myboxwidth=\halfwidth\fi
\vbox{%
\ifdim\myboxwidth=\hsize
\setbox\onelinebox=\hbox{%
\vbox{\hbox{%
$\Pi_{13,7}$ spans $L_{123.11}$%
}\hbox{%
$224\slashtwo422222|222\rtimes D_{2}$%
}%
}%
\hfill\copy\matricesbox
}%
\ifdim\wd\onelinebox>\myboxwidth
\hbox to \myboxwidth{%
$\Pi_{13,7}$ spans $L_{123.11}$%
\hfil
$224\slashtwo422222|222\rtimes D_{2}$%
}%
\box\matricesbox
\else
\hbox to \myboxwidth{%
\unhbox\onelinebox
}%
\fi
\else
\hbox to \myboxwidth{%
$\Pi_{13,7}$ spans $L_{123.11}$%
\hfil}%
\hbox to \myboxwidth{%
$224\slashtwo422222|222\rtimes D_{2}$%
\hfil}%
\box\matricesbox
\fi
}%
\hfill\discretionary{}{}{}%
\setbox\matricesbox=\hbox{%
{$\left[\!\llap{\phantom{%
\begingroup \smaller\smaller\smaller\begin{tabular}{@{}c@{}}%
\phantom{0}\\\phantom{0}\\\phantom{0}
\end{tabular}\endgroup%
}}\right.$}%
\begingroup \smaller\smaller\smaller\begin{tabular}{@{}c@{}}%
-1\\\phantom{0}\\\phantom{0}
\end{tabular}\endgroup%
\kern3pt%
\begingroup \smaller\smaller\smaller\begin{tabular}{@{}c@{}}%
\phantom{0}\\3\\\phantom{0}
\end{tabular}\endgroup%
\kern3pt%
\begingroup \smaller\smaller\smaller\begin{tabular}{@{}c@{}}%
\phantom{0}\\\phantom{0}\\3
\end{tabular}\endgroup%
{$\left.\llap{\phantom{%
\begingroup \smaller\smaller\smaller\begin{tabular}{@{}c@{}}%
\phantom{0}\\\phantom{0}\\\phantom{0}
\end{tabular}\endgroup%
}}\!\right]$}%
{$\left[\!\llap{\phantom{%
\begingroup \smaller\smaller\smaller\begin{tabular}{@{}c@{}}%
0\\0\\0
\end{tabular}\endgroup%
}}\right.$}%
\begingroup \smaller\smaller\smaller\begin{tabular}{@{}c@{}}%
24\\-14\\2
\end{tabular}\endgroup%
\kern3pt%
\begingroup \smaller\smaller\smaller\begin{tabular}{@{}c@{}}%
12\\-6\\4
\end{tabular}\endgroup%
\kern3pt%
\begingroup \smaller\smaller\smaller\begin{tabular}{@{}c@{}}%
12\\-4\\6
\end{tabular}\endgroup%
\kern3pt%
\begingroup \smaller\smaller\smaller\begin{tabular}{@{}c@{}}%
24\\-2\\14
\end{tabular}\endgroup%
\kern3pt%
\begingroup \smaller\smaller\smaller\begin{tabular}{@{}c@{}}%
24\\2\\14
\end{tabular}\endgroup%
\kern3pt%
\begingroup \smaller\smaller\smaller\begin{tabular}{@{}c@{}}%
2\\1\\1
\end{tabular}\endgroup%
\kern3pt%
\begingroup \smaller\smaller\smaller\begin{tabular}{@{}c@{}}%
3\\2\\0
\end{tabular}\endgroup%
{$\left.\llap{\phantom{%
\begingroup \smaller\smaller\smaller\begin{tabular}{@{}c@{}}%
0\\0\\0
\end{tabular}\endgroup%
}}\!\right]$}%
}%
\ifdim\wd\matricesbox>\halfwidth\myboxwidth=\hsize\else\myboxwidth=\halfwidth\fi
\vbox{%
\ifdim\myboxwidth=\hsize
\setbox\onelinebox=\hbox{%
\vbox{\hbox{%
$\Pi_{13,8}$ spans $L_{123.8}$%
}\hbox{%
$22424\slashtwo424222|2\rtimes D_{2}$%
}%
}%
\hfill\copy\matricesbox
}%
\ifdim\wd\onelinebox>\myboxwidth
\hbox to \myboxwidth{%
$\Pi_{13,8}$ spans $L_{123.8}$%
\hfil
$22424\slashtwo424222|2\rtimes D_{2}$%
}%
\box\matricesbox
\else
\hbox to \myboxwidth{%
\unhbox\onelinebox
}%
\fi
\else
\hbox to \myboxwidth{%
$\Pi_{13,8}$ spans $L_{123.8}$%
\hfil}%
\hbox to \myboxwidth{%
$22424\slashtwo424222|2\rtimes D_{2}$%
\hfil}%
\box\matricesbox
\fi
}%
\hfill\discretionary{}{}{}%
\setbox\matricesbox=\hbox{%
{$\left[\!\llap{\phantom{%
\begingroup \smaller\smaller\smaller\begin{tabular}{@{}c@{}}%
\phantom{0}\\\phantom{0}\\\phantom{0}
\end{tabular}\endgroup%
}}\right.$}%
\begingroup \smaller\smaller\smaller\begin{tabular}{@{}c@{}}%
-3\\\phantom{0}\\\phantom{0}
\end{tabular}\endgroup%
\kern3pt%
\begingroup \smaller\smaller\smaller\begin{tabular}{@{}c@{}}%
\phantom{0}\\4\\\phantom{0}
\end{tabular}\endgroup%
\kern3pt%
\begingroup \smaller\smaller\smaller\begin{tabular}{@{}c@{}}%
\phantom{0}\\\phantom{0}\\4
\end{tabular}\endgroup%
{$\left.\llap{\phantom{%
\begingroup \smaller\smaller\smaller\begin{tabular}{@{}c@{}}%
\phantom{0}\\\phantom{0}\\\phantom{0}
\end{tabular}\endgroup%
}}\!\right]$}%
{$\left[\!\llap{\phantom{%
\begingroup \smaller\smaller\smaller\begin{tabular}{@{}c@{}}%
0\\0\\0
\end{tabular}\endgroup%
}}\right.$}%
\begingroup \smaller\smaller\smaller\begin{tabular}{@{}c@{}}%
8\\-7\\1
\end{tabular}\endgroup%
\kern3pt%
\begingroup \smaller\smaller\smaller\begin{tabular}{@{}c@{}}%
4\\-3\\2
\end{tabular}\endgroup%
\kern3pt%
\begingroup \smaller\smaller\smaller\begin{tabular}{@{}c@{}}%
4\\-2\\3
\end{tabular}\endgroup%
\kern3pt%
\begingroup \smaller\smaller\smaller\begin{tabular}{@{}c@{}}%
1\\0\\1
\end{tabular}\endgroup%
\kern3pt%
\begingroup \smaller\smaller\smaller\begin{tabular}{@{}c@{}}%
4\\2\\3
\end{tabular}\endgroup%
\kern3pt%
\begingroup \smaller\smaller\smaller\begin{tabular}{@{}c@{}}%
4\\3\\2
\end{tabular}\endgroup%
\kern3pt%
\begingroup \smaller\smaller\smaller\begin{tabular}{@{}c@{}}%
1\\1\\0
\end{tabular}\endgroup%
{$\left.\llap{\phantom{%
\begingroup \smaller\smaller\smaller\begin{tabular}{@{}c@{}}%
0\\0\\0
\end{tabular}\endgroup%
}}\!\right]$}%
}%
\ifdim\wd\matricesbox>\halfwidth\myboxwidth=\hsize\else\myboxwidth=\halfwidth\fi
\vbox{%
\ifdim\myboxwidth=\hsize
\setbox\onelinebox=\hbox{%
\vbox{\hbox{%
$\Pi_{13,9}$ spans $L_{123.9}$%
}\hbox{%
$224\slashtwo422222|222\rtimes D_{2}$%
}%
}%
\hfill\copy\matricesbox
}%
\ifdim\wd\onelinebox>\myboxwidth
\hbox to \myboxwidth{%
$\Pi_{13,9}$ spans $L_{123.9}$%
\hfil
$224\slashtwo422222|222\rtimes D_{2}$%
}%
\box\matricesbox
\else
\hbox to \myboxwidth{%
\unhbox\onelinebox
}%
\fi
\else
\hbox to \myboxwidth{%
$\Pi_{13,9}$ spans $L_{123.9}$%
\hfil}%
\hbox to \myboxwidth{%
$224\slashtwo422222|222\rtimes D_{2}$%
\hfil}%
\box\matricesbox
\fi
}%
\hfill\discretionary{}{}{}%
\setbox\matricesbox=\hbox{%
{$\left[\!\llap{\phantom{%
\begingroup \smaller\smaller\smaller\begin{tabular}{@{}c@{}}%
\phantom{0}\\\phantom{0}\\\phantom{0}
\end{tabular}\endgroup%
}}\right.$}%
\begingroup \smaller\smaller\smaller\begin{tabular}{@{}c@{}}%
-1\\\phantom{0}\\\phantom{0}
\end{tabular}\endgroup%
\kern3pt%
\begingroup \smaller\smaller\smaller\begin{tabular}{@{}c@{}}%
\phantom{0}\\6\\\phantom{0}
\end{tabular}\endgroup%
\kern3pt%
\begingroup \smaller\smaller\smaller\begin{tabular}{@{}c@{}}%
\phantom{0}\\\phantom{0}\\6
\end{tabular}\endgroup%
{$\left.\llap{\phantom{%
\begingroup \smaller\smaller\smaller\begin{tabular}{@{}c@{}}%
\phantom{0}\\\phantom{0}\\\phantom{0}
\end{tabular}\endgroup%
}}\!\right]$}%
{$\left[\!\llap{\phantom{%
\begingroup \smaller\smaller\smaller\begin{tabular}{@{}c@{}}%
0\\0\\0
\end{tabular}\endgroup%
}}\right.$}%
\begingroup \smaller\smaller\smaller\begin{tabular}{@{}c@{}}%
12\\5\\-1
\end{tabular}\endgroup%
\kern3pt%
\begingroup \smaller\smaller\smaller\begin{tabular}{@{}c@{}}%
24\\8\\-6
\end{tabular}\endgroup%
\kern3pt%
\begingroup \smaller\smaller\smaller\begin{tabular}{@{}c@{}}%
24\\6\\-8
\end{tabular}\endgroup%
\kern3pt%
\begingroup \smaller\smaller\smaller\begin{tabular}{@{}c@{}}%
12\\1\\-5
\end{tabular}\endgroup%
\kern3pt%
\begingroup \smaller\smaller\smaller\begin{tabular}{@{}c@{}}%
12\\-1\\-5
\end{tabular}\endgroup%
\kern3pt%
\begingroup \smaller\smaller\smaller\begin{tabular}{@{}c@{}}%
3\\-1\\-1
\end{tabular}\endgroup%
\kern3pt%
\begingroup \smaller\smaller\smaller\begin{tabular}{@{}c@{}}%
2\\-1\\0
\end{tabular}\endgroup%
{$\left.\llap{\phantom{%
\begingroup \smaller\smaller\smaller\begin{tabular}{@{}c@{}}%
0\\0\\0
\end{tabular}\endgroup%
}}\!\right]$}%
}%
\ifdim\wd\matricesbox>\halfwidth\myboxwidth=\hsize\else\myboxwidth=\halfwidth\fi
\vbox{%
\ifdim\myboxwidth=\hsize
\setbox\onelinebox=\hbox{%
\vbox{\hbox{%
$\Pi_{13,10}$ spans $L_{123.8}$%
}\hbox{%
$22424\slashtwo424222|2\rtimes D_{2}$%
}%
}%
\hfill\copy\matricesbox
}%
\ifdim\wd\onelinebox>\myboxwidth
\hbox to \myboxwidth{%
$\Pi_{13,10}$ spans $L_{123.8}$%
\hfil
$22424\slashtwo424222|2\rtimes D_{2}$%
}%
\box\matricesbox
\else
\hbox to \myboxwidth{%
\unhbox\onelinebox
}%
\fi
\else
\hbox to \myboxwidth{%
$\Pi_{13,10}$ spans $L_{123.8}$%
\hfil}%
\hbox to \myboxwidth{%
$22424\slashtwo424222|2\rtimes D_{2}$%
\hfil}%
\box\matricesbox
\fi
}%
\hfill\discretionary{}{}{}%
\setbox\matricesbox=\hbox{%
{$\left[\!\llap{\phantom{%
\begingroup \smaller\smaller\smaller
\begin{tabular}{@{}c@{}}%
\phantom{0}\\\phantom{0}\\\phantom{0}
\end{tabular}\endgroup%
}}\right.$}%
\begingroup \smaller\smaller\smaller\begin{tabular}{@{}c@{}}%
-1\\\phantom{0}\\\phantom{0}
\end{tabular}\endgroup%
\kern3pt%
\begingroup \smaller\smaller\smaller\begin{tabular}{@{}c@{}}%
\phantom{0}\\6\\\phantom{0}
\end{tabular}\endgroup%
\kern3pt%
\begingroup \smaller\smaller\smaller\begin{tabular}{@{}c@{}}%
\phantom{0}\\\phantom{0}\\6
\end{tabular}\endgroup%
{$\left.\llap{\phantom{%
\begingroup \smaller\smaller\smaller\begin{tabular}{@{}c@{}}%
\phantom{0}\\\phantom{0}\\\phantom{0}
\end{tabular}\endgroup%
}}\!\right]$}%
{$\left[\!\llap{\phantom{%
\begingroup \smaller\smaller\smaller\begin{tabular}{@{}c@{}}%
0\\0\\0
\end{tabular}\endgroup%
}}\right.$}%
\begingroup \smaller\smaller\smaller\begin{tabular}{@{}c@{}}%
2\\1\\0
\end{tabular}\endgroup%
\kern3pt%
\begingroup \smaller\smaller\smaller\begin{tabular}{@{}c@{}}%
24\\8\\6
\end{tabular}\endgroup%
\kern3pt%
\begingroup \smaller\smaller\smaller\begin{tabular}{@{}c@{}}%
24\\6\\8
\end{tabular}\endgroup%
\kern3pt%
\begingroup \smaller\smaller\smaller\begin{tabular}{@{}c@{}}%
12\\1\\5
\end{tabular}\endgroup%
\kern3pt%
\begingroup \smaller\smaller\smaller\begin{tabular}{@{}c@{}}%
12\\-1\\5
\end{tabular}\endgroup%
\kern3pt%
\begingroup \smaller\smaller\smaller\begin{tabular}{@{}c@{}}%
3\\-1\\1
\end{tabular}\endgroup%
\kern3pt%
\begingroup \smaller\smaller\smaller\begin{tabular}{@{}c@{}}%
2\\-1\\0
\end{tabular}\endgroup%
\kern3pt%
\begingroup \smaller\smaller\smaller\begin{tabular}{@{}c@{}}%
24\\-8\\-6
\end{tabular}\endgroup%
\kern3pt%
\begingroup \smaller\smaller\smaller\begin{tabular}{@{}c@{}}%
24\\-6\\-8
\end{tabular}\endgroup%
\kern3pt%
\begingroup \smaller\smaller\smaller\begin{tabular}{@{}c@{}}%
12\\-1\\-5
\end{tabular}\endgroup%
\kern3pt%
\begingroup \smaller\smaller\smaller\begin{tabular}{@{}c@{}}%
12\\1\\-5
\end{tabular}\endgroup%
\kern3pt%
\begingroup \smaller\smaller\smaller\begin{tabular}{@{}c@{}}%
24\\6\\-8
\end{tabular}\endgroup%
\kern3pt%
\begingroup \smaller\smaller\smaller\begin{tabular}{@{}c@{}}%
24\\8\\-6
\end{tabular}\endgroup%
{$\left.\llap{\phantom{%
\begingroup \smaller\smaller\smaller\begin{tabular}{@{}c@{}}%
0\\0\\0
\end{tabular}\endgroup%
}}\!\right]$}%
}%
\ifdim\wd\matricesbox>\halfwidth\myboxwidth=\hsize\else\myboxwidth=\halfwidth\fi
\vbox{%
\ifdim\myboxwidth=\hsize
\setbox\onelinebox=\hbox{%
\vbox{\hbox{%
$\Pi_{13,11}$ spans $L_{123.8}$%
}\hbox{%
$2242222242422$%
}%
}%
\hfill\copy\matricesbox
}%
\ifdim\wd\onelinebox>\myboxwidth
\hbox to \myboxwidth{%
$\Pi_{13,11}$ spans $L_{123.8}$%
\hfil
$2242222242422$%
}%
\box\matricesbox
\else
\hbox to \myboxwidth{%
\unhbox\onelinebox
}%
\fi
\else
\hbox to \myboxwidth{%
$\Pi_{13,11}$ spans $L_{123.8}$%
\hfil}%
\hbox to \myboxwidth{%
$2242222242422$%
\hfil}%
\box\matricesbox
\fi
}%
\hfill\discretionary{}{}{}%
\setbox\matricesbox=\hbox{%
{$\left[\!\llap{\phantom{%
\begingroup \smaller\smaller\smaller
\begin{tabular}{@{}c@{}}%
\phantom{0}\\\phantom{0}\\\phantom{0}
\end{tabular}\endgroup%
}}\right.$}%
\begingroup \smaller\smaller\smaller\begin{tabular}{@{}c@{}}%
-1\\\phantom{0}\\\phantom{0}
\end{tabular}\endgroup%
\kern3pt%
\begingroup \smaller\smaller\smaller\begin{tabular}{@{}c@{}}%
\phantom{0}\\6\\\phantom{0}
\end{tabular}\endgroup%
\kern3pt%
\begingroup \smaller\smaller\smaller\begin{tabular}{@{}c@{}}%
\phantom{0}\\\phantom{0}\\6
\end{tabular}\endgroup%
{$\left.\llap{\phantom{%
\begingroup \smaller\smaller\smaller\begin{tabular}{@{}c@{}}%
\phantom{0}\\\phantom{0}\\\phantom{0}
\end{tabular}\endgroup%
}}\!\right]$}%
{$\left[\!\llap{\phantom{%
\begingroup \smaller\smaller\smaller\begin{tabular}{@{}c@{}}%
0\\0\\0
\end{tabular}\endgroup%
}}\right.$}%
\begingroup \smaller\smaller\smaller\begin{tabular}{@{}c@{}}%
2\\1\\0
\end{tabular}\endgroup%
\kern3pt%
\begingroup \smaller\smaller\smaller\begin{tabular}{@{}c@{}}%
24\\8\\6
\end{tabular}\endgroup%
\kern3pt%
\begingroup \smaller\smaller\smaller\begin{tabular}{@{}c@{}}%
24\\6\\8
\end{tabular}\endgroup%
\kern3pt%
\begingroup \smaller\smaller\smaller\begin{tabular}{@{}c@{}}%
12\\1\\5
\end{tabular}\endgroup%
\kern3pt%
\begingroup \smaller\smaller\smaller\begin{tabular}{@{}c@{}}%
12\\-1\\5
\end{tabular}\endgroup%
\kern3pt%
\begingroup \smaller\smaller\smaller\begin{tabular}{@{}c@{}}%
3\\-1\\1
\end{tabular}\endgroup%
\kern3pt%
\begingroup \smaller\smaller\smaller\begin{tabular}{@{}c@{}}%
12\\-5\\1
\end{tabular}\endgroup%
\kern3pt%
\begingroup \smaller\smaller\smaller\begin{tabular}{@{}c@{}}%
12\\-5\\-1
\end{tabular}\endgroup%
\kern3pt%
\begingroup \smaller\smaller\smaller\begin{tabular}{@{}c@{}}%
24\\-8\\-6
\end{tabular}\endgroup%
\kern3pt%
\begingroup \smaller\smaller\smaller\begin{tabular}{@{}c@{}}%
24\\-6\\-8
\end{tabular}\endgroup%
\kern3pt%
\begingroup \smaller\smaller\smaller\begin{tabular}{@{}c@{}}%
12\\-1\\-5
\end{tabular}\endgroup%
\kern3pt%
\begingroup \smaller\smaller\smaller\begin{tabular}{@{}c@{}}%
12\\1\\-5
\end{tabular}\endgroup%
\kern3pt%
\begingroup \smaller\smaller\smaller\begin{tabular}{@{}c@{}}%
3\\1\\-1
\end{tabular}\endgroup%
{$\left.\llap{\phantom{%
\begingroup \smaller\smaller\smaller\begin{tabular}{@{}c@{}}%
0\\0\\0
\end{tabular}\endgroup%
}}\!\right]$}%
}%
\ifdim\wd\matricesbox>\halfwidth\myboxwidth=\hsize\else\myboxwidth=\halfwidth\fi
\vbox{%
\ifdim\myboxwidth=\hsize
\setbox\onelinebox=\hbox{%
\vbox{\hbox{%
$\Pi_{13,12}$ spans $L_{123.8}$%
}\hbox{%
$2242222424222$%
}%
}%
\hfill\copy\matricesbox
}%
\ifdim\wd\onelinebox>\myboxwidth
\hbox to \myboxwidth{%
$\Pi_{13,12}$ spans $L_{123.8}$%
\hfil
$2242222424222$%
}%
\box\matricesbox
\else
\hbox to \myboxwidth{%
\unhbox\onelinebox
}%
\fi
\else
\hbox to \myboxwidth{%
$\Pi_{13,12}$ spans $L_{123.8}$%
\hfil}%
\hbox to \myboxwidth{%
$2242222424222$%
\hfil}%
\box\matricesbox
\fi
}%
\hfill\discretionary{}{}{}%
\setbox\matricesbox=\hbox{%
{$\left[\!\llap{\phantom{%
\begingroup \smaller\smaller\smaller
\begin{tabular}{@{}c@{}}%
\phantom{0}\\\phantom{0}\\\phantom{0}
\end{tabular}\endgroup%
}}\right.$}%
\begingroup \smaller\smaller\smaller\begin{tabular}{@{}c@{}}%
-1\\\phantom{0}\\\phantom{0}
\end{tabular}\endgroup%
\kern3pt%
\begingroup \smaller\smaller\smaller\begin{tabular}{@{}c@{}}%
\phantom{0}\\6\\\phantom{0}
\end{tabular}\endgroup%
\kern3pt%
\begingroup \smaller\smaller\smaller\begin{tabular}{@{}c@{}}%
\phantom{0}\\\phantom{0}\\6
\end{tabular}\endgroup%
{$\left.\llap{\phantom{%
\begingroup \smaller\smaller\smaller\begin{tabular}{@{}c@{}}%
\phantom{0}\\\phantom{0}\\\phantom{0}
\end{tabular}\endgroup%
}}\!\right]$}%
{$\left[\!\llap{\phantom{%
\begingroup \smaller\smaller\smaller\begin{tabular}{@{}c@{}}%
0\\0\\0
\end{tabular}\endgroup%
}}\right.$}%
\begingroup \smaller\smaller\smaller\begin{tabular}{@{}c@{}}%
2\\1\\0
\end{tabular}\endgroup%
\kern3pt%
\begingroup \smaller\smaller\smaller\begin{tabular}{@{}c@{}}%
24\\8\\6
\end{tabular}\endgroup%
\kern3pt%
\begingroup \smaller\smaller\smaller\begin{tabular}{@{}c@{}}%
24\\6\\8
\end{tabular}\endgroup%
\kern3pt%
\begingroup \smaller\smaller\smaller\begin{tabular}{@{}c@{}}%
12\\1\\5
\end{tabular}\endgroup%
\kern3pt%
\begingroup \smaller\smaller\smaller\begin{tabular}{@{}c@{}}%
12\\-1\\5
\end{tabular}\endgroup%
\kern3pt%
\begingroup \smaller\smaller\smaller\begin{tabular}{@{}c@{}}%
24\\-6\\8
\end{tabular}\endgroup%
\kern3pt%
\begingroup \smaller\smaller\smaller\begin{tabular}{@{}c@{}}%
24\\-8\\6
\end{tabular}\endgroup%
\kern3pt%
\begingroup \smaller\smaller\smaller\begin{tabular}{@{}c@{}}%
12\\-5\\1
\end{tabular}\endgroup%
\kern3pt%
\begingroup \smaller\smaller\smaller\begin{tabular}{@{}c@{}}%
12\\-5\\-1
\end{tabular}\endgroup%
\kern3pt%
\begingroup \smaller\smaller\smaller\begin{tabular}{@{}c@{}}%
3\\-1\\-1
\end{tabular}\endgroup%
\kern3pt%
\begingroup \smaller\smaller\smaller\begin{tabular}{@{}c@{}}%
2\\0\\-1
\end{tabular}\endgroup%
\kern3pt%
\begingroup \smaller\smaller\smaller\begin{tabular}{@{}c@{}}%
24\\6\\-8
\end{tabular}\endgroup%
\kern3pt%
\begingroup \smaller\smaller\smaller\begin{tabular}{@{}c@{}}%
24\\8\\-6
\end{tabular}\endgroup%
{$\left.\llap{\phantom{%
\begingroup \smaller\smaller\smaller\begin{tabular}{@{}c@{}}%
0\\0\\0
\end{tabular}\endgroup%
}}\!\right]$}%
}%
\ifdim\wd\matricesbox>\halfwidth\myboxwidth=\hsize\else\myboxwidth=\halfwidth\fi
\vbox{%
\ifdim\myboxwidth=\hsize
\setbox\onelinebox=\hbox{%
\vbox{\hbox{%
$\Pi_{13,13}$ spans $L_{123.8}$%
}\hbox{%
$2242424222222$%
}%
}%
\hfill\copy\matricesbox
}%
\ifdim\wd\onelinebox>\myboxwidth
\hbox to \myboxwidth{%
$\Pi_{13,13}$ spans $L_{123.8}$%
\hfil
$2242424222222$%
}%
\box\matricesbox
\else
\hbox to \myboxwidth{%
\unhbox\onelinebox
}%
\fi
\else
\hbox to \myboxwidth{%
$\Pi_{13,13}$ spans $L_{123.8}$%
\hfil}%
\hbox to \myboxwidth{%
$2242424222222$%
\hfil}%
\box\matricesbox
\fi
}%
\hfill\discretionary{}{}{}%
\setbox\matricesbox=\hbox{%
{$\left[\!\llap{\phantom{%
\begingroup \smaller\smaller\smaller
\begin{tabular}{@{}c@{}}%
\phantom{0}\\\phantom{0}\\\phantom{0}
\end{tabular}\endgroup%
}}\right.$}%
\begingroup \smaller\smaller\smaller\begin{tabular}{@{}c@{}}%
-1\\\phantom{0}\\\phantom{0}
\end{tabular}\endgroup%
\kern3pt%
\begingroup \smaller\smaller\smaller\begin{tabular}{@{}c@{}}%
\phantom{0}\\6\\\phantom{0}
\end{tabular}\endgroup%
\kern3pt%
\begingroup \smaller\smaller\smaller\begin{tabular}{@{}c@{}}%
\phantom{0}\\\phantom{0}\\6
\end{tabular}\endgroup%
{$\left.\llap{\phantom{%
\begingroup \smaller\smaller\smaller\begin{tabular}{@{}c@{}}%
\phantom{0}\\\phantom{0}\\\phantom{0}
\end{tabular}\endgroup%
}}\!\right]$}%
{$\left[\!\llap{\phantom{%
\begingroup \smaller\smaller\smaller\begin{tabular}{@{}c@{}}%
0\\0\\0
\end{tabular}\endgroup%
}}\right.$}%
\begingroup \smaller\smaller\smaller\begin{tabular}{@{}c@{}}%
2\\1\\0
\end{tabular}\endgroup%
\kern3pt%
\begingroup \smaller\smaller\smaller\begin{tabular}{@{}c@{}}%
24\\8\\6
\end{tabular}\endgroup%
\kern3pt%
\begingroup \smaller\smaller\smaller\begin{tabular}{@{}c@{}}%
24\\6\\8
\end{tabular}\endgroup%
\kern3pt%
\begingroup \smaller\smaller\smaller\begin{tabular}{@{}c@{}}%
12\\1\\5
\end{tabular}\endgroup%
\kern3pt%
\begingroup \smaller\smaller\smaller\begin{tabular}{@{}c@{}}%
12\\-1\\5
\end{tabular}\endgroup%
\kern3pt%
\begingroup \smaller\smaller\smaller\begin{tabular}{@{}c@{}}%
24\\-6\\8
\end{tabular}\endgroup%
\kern3pt%
\begingroup \smaller\smaller\smaller\begin{tabular}{@{}c@{}}%
24\\-8\\6
\end{tabular}\endgroup%
\kern3pt%
\begingroup \smaller\smaller\smaller\begin{tabular}{@{}c@{}}%
12\\-5\\1
\end{tabular}\endgroup%
\kern3pt%
\begingroup \smaller\smaller\smaller\begin{tabular}{@{}c@{}}%
12\\-5\\-1
\end{tabular}\endgroup%
\kern3pt%
\begingroup \smaller\smaller\smaller\begin{tabular}{@{}c@{}}%
3\\-1\\-1
\end{tabular}\endgroup%
\kern3pt%
\begingroup \smaller\smaller\smaller\begin{tabular}{@{}c@{}}%
12\\-1\\-5
\end{tabular}\endgroup%
\kern3pt%
\begingroup \smaller\smaller\smaller\begin{tabular}{@{}c@{}}%
12\\1\\-5
\end{tabular}\endgroup%
\kern3pt%
\begingroup \smaller\smaller\smaller\begin{tabular}{@{}c@{}}%
3\\1\\-1
\end{tabular}\endgroup%
{$\left.\llap{\phantom{%
\begingroup \smaller\smaller\smaller\begin{tabular}{@{}c@{}}%
0\\0\\0
\end{tabular}\endgroup%
}}\!\right]$}%
}%
\ifdim\wd\matricesbox>\halfwidth\myboxwidth=\hsize\else\myboxwidth=\halfwidth\fi
\vbox{%
\ifdim\myboxwidth=\hsize
\setbox\onelinebox=\hbox{%
\vbox{\hbox{%
$\Pi_{13,14}$ spans $L_{123.8}$%
}\hbox{%
$2242424222222$%
}%
}%
\hfill\copy\matricesbox
}%
\ifdim\wd\onelinebox>\myboxwidth
\hbox to \myboxwidth{%
$\Pi_{13,14}$ spans $L_{123.8}$%
\hfil
$2242424222222$%
}%
\box\matricesbox
\else
\hbox to \myboxwidth{%
\unhbox\onelinebox
}%
\fi
\else
\hbox to \myboxwidth{%
$\Pi_{13,14}$ spans $L_{123.8}$%
\hfil}%
\hbox to \myboxwidth{%
$2242424222222$%
\hfil}%
\box\matricesbox
\fi
}%
\hfill\discretionary{}{}{}%
\setbox\matricesbox=\hbox{%
{$\left[\!\llap{\phantom{%
\begingroup \smaller\smaller\smaller
\begin{tabular}{@{}c@{}}%
\phantom{0}\\\phantom{0}\\\phantom{0}
\end{tabular}\endgroup%
}}\right.$}%
\begingroup \smaller\smaller\smaller\begin{tabular}{@{}c@{}}%
-1\\\phantom{0}\\\phantom{0}
\end{tabular}\endgroup%
\kern3pt%
\begingroup \smaller\smaller\smaller\begin{tabular}{@{}c@{}}%
\phantom{0}\\18\\\phantom{0}
\end{tabular}\endgroup%
\kern3pt%
\begingroup \smaller\smaller\smaller\begin{tabular}{@{}c@{}}%
\phantom{0}\\\phantom{0}\\18
\end{tabular}\endgroup%
{$\left.\llap{\phantom{%
\begingroup \smaller\smaller\smaller\begin{tabular}{@{}c@{}}%
\phantom{0}\\\phantom{0}\\\phantom{0}
\end{tabular}\endgroup%
}}\!\right]$}%
{$\left[\!\llap{\phantom{%
\begingroup \smaller\smaller\smaller\begin{tabular}{@{}c@{}}%
0\\0\\0
\end{tabular}\endgroup%
}}\right.$}%
\begingroup \smaller\smaller\smaller\begin{tabular}{@{}c@{}}%
9\\1\\2
\end{tabular}\endgroup%
\kern3pt%
\begingroup \smaller\smaller\smaller\begin{tabular}{@{}c@{}}%
36\\7\\5
\end{tabular}\endgroup%
\kern3pt%
\begingroup \smaller\smaller\smaller\begin{tabular}{@{}c@{}}%
72\\16\\6
\end{tabular}\endgroup%
\kern3pt%
\begingroup \smaller\smaller\smaller\begin{tabular}{@{}c@{}}%
8\\2\\0
\end{tabular}\endgroup%
\kern3pt%
\begingroup \smaller\smaller\smaller\begin{tabular}{@{}c@{}}%
9\\2\\-1
\end{tabular}\endgroup%
\kern3pt%
\begingroup \smaller\smaller\smaller\begin{tabular}{@{}c@{}}%
9\\1\\-2
\end{tabular}\endgroup%
\kern3pt%
\begingroup \smaller\smaller\smaller\begin{tabular}{@{}c@{}}%
8\\0\\-2
\end{tabular}\endgroup%
\kern3pt%
\begingroup \smaller\smaller\smaller\begin{tabular}{@{}c@{}}%
9\\-1\\-2
\end{tabular}\endgroup%
\kern3pt%
\begingroup \smaller\smaller\smaller\begin{tabular}{@{}c@{}}%
9\\-2\\-1
\end{tabular}\endgroup%
\kern3pt%
\begingroup \smaller\smaller\smaller\begin{tabular}{@{}c@{}}%
8\\-2\\0
\end{tabular}\endgroup%
\kern3pt%
\begingroup \smaller\smaller\smaller\begin{tabular}{@{}c@{}}%
9\\-2\\1
\end{tabular}\endgroup%
\kern3pt%
\begingroup \smaller\smaller\smaller\begin{tabular}{@{}c@{}}%
9\\-1\\2
\end{tabular}\endgroup%
\kern3pt%
\begingroup \smaller\smaller\smaller\begin{tabular}{@{}c@{}}%
8\\0\\2
\end{tabular}\endgroup%
{$\left.\llap{\phantom{%
\begingroup \smaller\smaller\smaller\begin{tabular}{@{}c@{}}%
0\\0\\0
\end{tabular}\endgroup%
}}\!\right]$}%
}%
\ifdim\wd\matricesbox>\halfwidth\myboxwidth=\hsize\else\myboxwidth=\halfwidth\fi
\vbox{%
\ifdim\myboxwidth=\hsize
\setbox\onelinebox=\hbox{%
\vbox{\hbox{%
$\Pi_{13,15}$ spans $L_{142.20}$%
}\hbox{%
$\infty422\infty22\infty22\infty22$%
}%
}%
\hfill\copy\matricesbox
}%
\ifdim\wd\onelinebox>\myboxwidth
\hbox to \myboxwidth{%
$\Pi_{13,15}$ spans $L_{142.20}$%
\hfil
$\infty422\infty22\infty22\infty22$%
}%
\box\matricesbox
\else
\hbox to \myboxwidth{%
\unhbox\onelinebox
}%
\fi
\else
\hbox to \myboxwidth{%
$\Pi_{13,15}$ spans $L_{142.20}$%
\hfil}%
\hbox to \myboxwidth{%
$\infty422\infty22\infty22\infty22$%
\hfil}%
\box\matricesbox
\fi
}%
\hfill\discretionary{}{}{}%

\vskip2pt\hrule\vskip2pt

\leavevmode\setbox\matricesbox=\hbox{%
{$\left[\!\llap{\phantom{%
\begingroup \smaller\smaller\smaller\begin{tabular}{@{}c@{}}%
\phantom{0}\\\phantom{0}\\\phantom{0}
\end{tabular}\endgroup%
}}\right.$}%
\begingroup \smaller\smaller\smaller\begin{tabular}{@{}c@{}}%
-1\\\phantom{0}\\\phantom{0}
\end{tabular}\endgroup%
\kern3pt%
\begingroup \smaller\smaller\smaller\begin{tabular}{@{}c@{}}%
\phantom{0}\\15/2\\\phantom{0}
\end{tabular}\endgroup%
\kern3pt%
\begingroup \smaller\smaller\smaller\begin{tabular}{@{}c@{}}%
\phantom{0}\\\phantom{0}\\45/2
\end{tabular}\endgroup%
{$\left.\llap{\phantom{%
\begingroup \smaller\smaller\smaller\begin{tabular}{@{}c@{}}%
\phantom{0}\\\phantom{0}\\\phantom{0}
\end{tabular}\endgroup%
}}\!\right]$}%
{$\left[\!\llap{\phantom{%
\begingroup \smaller\smaller\smaller\begin{tabular}{@{}c@{}}%
0\\0\\0
\end{tabular}\endgroup%
}}\right.$}%
\begingroup \smaller\smaller\smaller\begin{tabular}{@{}c@{}}%
5\\-2\\0
\end{tabular}\endgroup%
\kern3pt%
\begingroup \smaller\smaller\smaller\begin{tabular}{@{}c@{}}%
9\\-3\\1
\end{tabular}\endgroup%
\kern3pt%
\begingroup \smaller\smaller\smaller\begin{tabular}{@{}c@{}}%
5\\-1\\1
\end{tabular}\endgroup%
\kern3pt%
\begingroup \smaller\smaller\smaller\begin{tabular}{@{}c@{}}%
90\\-3\\19
\end{tabular}\endgroup%
{$\left.\llap{\phantom{%
\begingroup \smaller\smaller\smaller\begin{tabular}{@{}c@{}}%
0\\0\\0
\end{tabular}\endgroup%
}}\!\right]$}%
}%
\ifdim\wd\matricesbox>\halfwidth\myboxwidth=\hsize\else\myboxwidth=\halfwidth\fi
\vbox{%
\ifdim\myboxwidth=\hsize
\setbox\onelinebox=\hbox{%
\vbox{\hbox{%
$\Pi_{14,1}$ spans $L_{16.13}$%
}\hbox{%
$22|222\slashthree222|222\slashthree2\rtimes D_{4}$%
}%
}%
\hfill\copy\matricesbox
}%
\ifdim\wd\onelinebox>\myboxwidth
\hbox to \myboxwidth{%
$\Pi_{14,1}$ spans $L_{16.13}$%
\hfil
$22|222\slashthree222|222\slashthree2\rtimes D_{4}$%
}%
\box\matricesbox
\else
\hbox to \myboxwidth{%
\unhbox\onelinebox
}%
\fi
\else
\hbox to \myboxwidth{%
$\Pi_{14,1}$ spans $L_{16.13}$%
\hfil}%
\hbox to \myboxwidth{%
$22|222\slashthree222|222\slashthree2\rtimes D_{4}$%
\hfil}%
\box\matricesbox
\fi
}%
\hfill\discretionary{}{}{}%
\setbox\matricesbox=\hbox{%
{$\left[\!\llap{\phantom{%
\begingroup \smaller\smaller\smaller\begin{tabular}{@{}c@{}}%
\phantom{0}\\\phantom{0}\\\phantom{0}
\end{tabular}\endgroup%
}}\right.$}%
\begingroup \smaller\smaller\smaller\begin{tabular}{@{}c@{}}%
-2\\\phantom{0}\\\phantom{0}
\end{tabular}\endgroup%
\kern3pt%
\begingroup \smaller\smaller\smaller\begin{tabular}{@{}c@{}}%
\phantom{0}\\3\\\phantom{0}
\end{tabular}\endgroup%
\kern3pt%
\begingroup \smaller\smaller\smaller\begin{tabular}{@{}c@{}}%
\phantom{0}\\\phantom{0}\\3
\end{tabular}\endgroup%
{$\left.\llap{\phantom{%
\begingroup \smaller\smaller\smaller\begin{tabular}{@{}c@{}}%
\phantom{0}\\\phantom{0}\\\phantom{0}
\end{tabular}\endgroup%
}}\!\right]$}%
{$\left[\!\llap{\phantom{%
\begingroup \smaller\smaller\smaller\begin{tabular}{@{}c@{}}%
0\\0\\0
\end{tabular}\endgroup%
}}\right.$}%
\begingroup \smaller\smaller\smaller\begin{tabular}{@{}c@{}}%
1\\1\\0
\end{tabular}\endgroup%
\kern3pt%
\begingroup \smaller\smaller\smaller\begin{tabular}{@{}c@{}}%
12\\8\\-6
\end{tabular}\endgroup%
\kern3pt%
\begingroup \smaller\smaller\smaller\begin{tabular}{@{}c@{}}%
12\\6\\-8
\end{tabular}\endgroup%
\kern3pt%
\begingroup \smaller\smaller\smaller\begin{tabular}{@{}c@{}}%
6\\1\\-5
\end{tabular}\endgroup%
{$\left.\llap{\phantom{%
\begingroup \smaller\smaller\smaller\begin{tabular}{@{}c@{}}%
0\\0\\0
\end{tabular}\endgroup%
}}\!\right]$}%
}%
\ifdim\wd\matricesbox>\halfwidth\myboxwidth=\hsize\else\myboxwidth=\halfwidth\fi
\vbox{%
\ifdim\myboxwidth=\hsize
\setbox\onelinebox=\hbox{%
\vbox{\hbox{%
$\Pi_{14,2}$ spans $L_{123.11}$%
}\hbox{%
$|224\slashtwo422|224\slashtwo422\rtimes D_{4}$%
}%
}%
\hfill\copy\matricesbox
}%
\ifdim\wd\onelinebox>\myboxwidth
\hbox to \myboxwidth{%
$\Pi_{14,2}$ spans $L_{123.11}$%
\hfil
$|224\slashtwo422|224\slashtwo422\rtimes D_{4}$%
}%
\box\matricesbox
\else
\hbox to \myboxwidth{%
\unhbox\onelinebox
}%
\fi
\else
\hbox to \myboxwidth{%
$\Pi_{14,2}$ spans $L_{123.11}$%
\hfil}%
\hbox to \myboxwidth{%
$|224\slashtwo422|224\slashtwo422\rtimes D_{4}$%
\hfil}%
\box\matricesbox
\fi
}%
\hfill\discretionary{}{}{}%
\setbox\matricesbox=\hbox{%
{$\left[\!\llap{\phantom{%
\begingroup \smaller\smaller\smaller\begin{tabular}{@{}c@{}}%
\phantom{0}\\\phantom{0}\\\phantom{0}
\end{tabular}\endgroup%
}}\right.$}%
\begingroup \smaller\smaller\smaller\begin{tabular}{@{}c@{}}%
-3\\\phantom{0}\\\phantom{0}
\end{tabular}\endgroup%
\kern3pt%
\begingroup \smaller\smaller\smaller\begin{tabular}{@{}c@{}}%
\phantom{0}\\4\\\phantom{0}
\end{tabular}\endgroup%
\kern3pt%
\begingroup \smaller\smaller\smaller\begin{tabular}{@{}c@{}}%
\phantom{0}\\\phantom{0}\\4
\end{tabular}\endgroup%
{$\left.\llap{\phantom{%
\begingroup \smaller\smaller\smaller\begin{tabular}{@{}c@{}}%
\phantom{0}\\\phantom{0}\\\phantom{0}
\end{tabular}\endgroup%
}}\!\right]$}%
{$\left[\!\llap{\phantom{%
\begingroup \smaller\smaller\smaller\begin{tabular}{@{}c@{}}%
0\\0\\0
\end{tabular}\endgroup%
}}\right.$}%
\begingroup \smaller\smaller\smaller\begin{tabular}{@{}c@{}}%
1\\-1\\0
\end{tabular}\endgroup%
\kern3pt%
\begingroup \smaller\smaller\smaller\begin{tabular}{@{}c@{}}%
4\\-3\\2
\end{tabular}\endgroup%
\kern3pt%
\begingroup \smaller\smaller\smaller\begin{tabular}{@{}c@{}}%
4\\-2\\3
\end{tabular}\endgroup%
\kern3pt%
\begingroup \smaller\smaller\smaller\begin{tabular}{@{}c@{}}%
8\\-1\\7
\end{tabular}\endgroup%
{$\left.\llap{\phantom{%
\begingroup \smaller\smaller\smaller\begin{tabular}{@{}c@{}}%
0\\0\\0
\end{tabular}\endgroup%
}}\!\right]$}%
}%
\ifdim\wd\matricesbox>\halfwidth\myboxwidth=\hsize\else\myboxwidth=\halfwidth\fi
\vbox{%
\ifdim\myboxwidth=\hsize
\setbox\onelinebox=\hbox{%
\vbox{\hbox{%
$\Pi_{14,3}$ spans $L_{123.9}$%
}\hbox{%
$|224\slashtwo422|224\slashtwo422\rtimes D_{4}$%
}%
}%
\hfill\copy\matricesbox
}%
\ifdim\wd\onelinebox>\myboxwidth
\hbox to \myboxwidth{%
$\Pi_{14,3}$ spans $L_{123.9}$%
\hfil
$|224\slashtwo422|224\slashtwo422\rtimes D_{4}$%
}%
\box\matricesbox
\else
\hbox to \myboxwidth{%
\unhbox\onelinebox
}%
\fi
\else
\hbox to \myboxwidth{%
$\Pi_{14,3}$ spans $L_{123.9}$%
\hfil}%
\hbox to \myboxwidth{%
$|224\slashtwo422|224\slashtwo422\rtimes D_{4}$%
\hfil}%
\box\matricesbox
\fi
}%
\hfill\discretionary{}{}{}%
\setbox\matricesbox=\hbox{%
{$\left[\!\llap{\phantom{%
\begingroup \smaller\smaller\smaller\begin{tabular}{@{}c@{}}%
\phantom{0}\\\phantom{0}\\\phantom{0}
\end{tabular}\endgroup%
}}\right.$}%
\begingroup \smaller\smaller\smaller\begin{tabular}{@{}c@{}}%
-1\\\phantom{0}\\\phantom{0}
\end{tabular}\endgroup%
\kern3pt%
\begingroup \smaller\smaller\smaller\begin{tabular}{@{}c@{}}%
\phantom{0}\\5\\\phantom{0}
\end{tabular}\endgroup%
\kern3pt%
\begingroup \smaller\smaller\smaller\begin{tabular}{@{}c@{}}%
\phantom{0}\\\phantom{0}\\15
\end{tabular}\endgroup%
{$\left.\llap{\phantom{%
\begingroup \smaller\smaller\smaller\begin{tabular}{@{}c@{}}%
\phantom{0}\\\phantom{0}\\\phantom{0}
\end{tabular}\endgroup%
}}\!\right]$}%
{$\left[\!\llap{\phantom{%
\begingroup \smaller\smaller\smaller\begin{tabular}{@{}c@{}}%
0\\0\\0
\end{tabular}\endgroup%
}}\right.$}%
\begingroup \smaller\smaller\smaller\begin{tabular}{@{}c@{}}%
20\\9\\1
\end{tabular}\endgroup%
\kern3pt%
\begingroup \smaller\smaller\smaller\begin{tabular}{@{}c@{}}%
15\\6\\2
\end{tabular}\endgroup%
\kern3pt%
\begingroup \smaller\smaller\smaller\begin{tabular}{@{}c@{}}%
4\\1\\1
\end{tabular}\endgroup%
\kern3pt%
\begingroup \smaller\smaller\smaller\begin{tabular}{@{}c@{}}%
15\\0\\4
\end{tabular}\endgroup%
{$\left.\llap{\phantom{%
\begingroup \smaller\smaller\smaller\begin{tabular}{@{}c@{}}%
0\\0\\0
\end{tabular}\endgroup%
}}\!\right]$}%
}%
\ifdim\wd\matricesbox>\halfwidth\myboxwidth=\hsize\else\myboxwidth=\halfwidth\fi
\vbox{%
\ifdim\myboxwidth=\hsize
\setbox\onelinebox=\hbox{%
\vbox{\hbox{%
$\Pi_{14,4}$ spans $L_{31.7}$%
}\hbox{%
$22\slashthree222|222\slashthree222|2\rtimes D_{4}$%
}%
}%
\hfill\copy\matricesbox
}%
\ifdim\wd\onelinebox>\myboxwidth
\hbox to \myboxwidth{%
$\Pi_{14,4}$ spans $L_{31.7}$%
\hfil
$22\slashthree222|222\slashthree222|2\rtimes D_{4}$%
}%
\box\matricesbox
\else
\hbox to \myboxwidth{%
\unhbox\onelinebox
}%
\fi
\else
\hbox to \myboxwidth{%
$\Pi_{14,4}$ spans $L_{31.7}$%
\hfil}%
\hbox to \myboxwidth{%
$22\slashthree222|222\slashthree222|2\rtimes D_{4}$%
\hfil}%
\box\matricesbox
\fi
}%
\hfill\discretionary{}{}{}%
\setbox\matricesbox=\hbox{%
{$\left[\!\llap{\phantom{%
\begingroup \smaller\smaller\smaller\begin{tabular}{@{}c@{}}%
\phantom{0}\\\phantom{0}\\\phantom{0}
\end{tabular}\endgroup%
}}\right.$}%
\begingroup \smaller\smaller\smaller\begin{tabular}{@{}c@{}}%
-1\\\phantom{0}\\\phantom{0}
\end{tabular}\endgroup%
\kern3pt%
\begingroup \smaller\smaller\smaller\begin{tabular}{@{}c@{}}%
\phantom{0}\\7/2\\\phantom{0}
\end{tabular}\endgroup%
\kern3pt%
\begingroup \smaller\smaller\smaller\begin{tabular}{@{}c@{}}%
\phantom{0}\\\phantom{0}\\21/2
\end{tabular}\endgroup%
{$\left.\llap{\phantom{%
\begingroup \smaller\smaller\smaller\begin{tabular}{@{}c@{}}%
\phantom{0}\\\phantom{0}\\\phantom{0}
\end{tabular}\endgroup%
}}\!\right]$}%
{$\left[\!\llap{\phantom{%
\begingroup \smaller\smaller\smaller\begin{tabular}{@{}c@{}}%
0\\0\\0
\end{tabular}\endgroup%
}}\right.$}%
\begingroup \smaller\smaller\smaller\begin{tabular}{@{}c@{}}%
7\\-4\\0
\end{tabular}\endgroup%
\kern3pt%
\begingroup \smaller\smaller\smaller\begin{tabular}{@{}c@{}}%
6\\-3\\-1
\end{tabular}\endgroup%
\kern3pt%
\begingroup \smaller\smaller\smaller\begin{tabular}{@{}c@{}}%
7\\-2\\-2
\end{tabular}\endgroup%
\kern3pt%
\begingroup \smaller\smaller\smaller\begin{tabular}{@{}c@{}}%
42\\-3\\-13
\end{tabular}\endgroup%
{$\left.\llap{\phantom{%
\begingroup \smaller\smaller\smaller\begin{tabular}{@{}c@{}}%
0\\0\\0
\end{tabular}\endgroup%
}}\!\right]$}%
}%
\ifdim\wd\matricesbox>\halfwidth\myboxwidth=\hsize\else\myboxwidth=\halfwidth\fi
\vbox{%
\ifdim\myboxwidth=\hsize
\setbox\onelinebox=\hbox{%
\vbox{\hbox{%
$\Pi_{14,5}$ spans $L_{22.4}$%
}\hbox{%
$2|222\slashthree222|222\slashthree22\rtimes D_{4}$%
}%
}%
\hfill\copy\matricesbox
}%
\ifdim\wd\onelinebox>\myboxwidth
\hbox to \myboxwidth{%
$\Pi_{14,5}$ spans $L_{22.4}$%
\hfil
$2|222\slashthree222|222\slashthree22\rtimes D_{4}$%
}%
\box\matricesbox
\else
\hbox to \myboxwidth{%
\unhbox\onelinebox
}%
\fi
\else
\hbox to \myboxwidth{%
$\Pi_{14,5}$ spans $L_{22.4}$%
\hfil}%
\hbox to \myboxwidth{%
$2|222\slashthree222|222\slashthree22\rtimes D_{4}$%
\hfil}%
\box\matricesbox
\fi
}%
\hfill\discretionary{}{}{}%
\setbox\matricesbox=\hbox{%
{$\left[\!\llap{\phantom{%
\begingroup \smaller\smaller\smaller\begin{tabular}{@{}c@{}}%
\phantom{0}\\\phantom{0}\\\phantom{0}
\end{tabular}\endgroup%
}}\right.$}%
\begingroup \smaller\smaller\smaller\begin{tabular}{@{}c@{}}%
-3/2\\\phantom{0}\\\phantom{0}
\end{tabular}\endgroup%
\kern3pt%
\begingroup \smaller\smaller\smaller\begin{tabular}{@{}c@{}}%
\phantom{0}\\5/2\\\phantom{0}
\end{tabular}\endgroup%
\kern3pt%
\begingroup \smaller\smaller\smaller\begin{tabular}{@{}c@{}}%
\phantom{0}\\\phantom{0}\\75/2
\end{tabular}\endgroup%
{$\left.\llap{\phantom{%
\begingroup \smaller\smaller\smaller\begin{tabular}{@{}c@{}}%
\phantom{0}\\\phantom{0}\\\phantom{0}
\end{tabular}\endgroup%
}}\!\right]$}%
{$\left[\!\llap{\phantom{%
\begingroup \smaller\smaller\smaller\begin{tabular}{@{}c@{}}%
0\\0\\0
\end{tabular}\endgroup%
}}\right.$}%
\begingroup \smaller\smaller\smaller\begin{tabular}{@{}c@{}}%
10\\8\\0
\end{tabular}\endgroup%
\kern3pt%
\begingroup \smaller\smaller\smaller\begin{tabular}{@{}c@{}}%
10\\7\\1
\end{tabular}\endgroup%
\kern3pt%
\begingroup \smaller\smaller\smaller\begin{tabular}{@{}c@{}}%
6\\3\\1
\end{tabular}\endgroup%
\kern3pt%
\begingroup \smaller\smaller\smaller\begin{tabular}{@{}c@{}}%
10\\2\\2
\end{tabular}\endgroup%
{$\left.\llap{\phantom{%
\begingroup \smaller\smaller\smaller\begin{tabular}{@{}c@{}}%
0\\0\\0
\end{tabular}\endgroup%
}}\!\right]$}%
}%
\ifdim\wd\matricesbox>\halfwidth\myboxwidth=\hsize\else\myboxwidth=\halfwidth\fi
\vbox{%
\ifdim\myboxwidth=\hsize
\setbox\onelinebox=\hbox{%
\vbox{\hbox{%
$\Pi_{14,6}$ spans $L_{30.27}$%
}\hbox{%
$\infty|\infty22\slashinfty22\infty|\infty22\slashinfty22\rtimes D_{4}$%
}%
}%
\hfill\copy\matricesbox
}%
\ifdim\wd\onelinebox>\myboxwidth
\hbox to \myboxwidth{%
$\Pi_{14,6}$ spans $L_{30.27}$%
\hfil
$\infty|\infty22\slashinfty22\infty|\infty22\slashinfty22\rtimes D_{4}$%
}%
\box\matricesbox
\else
\hbox to \myboxwidth{%
\unhbox\onelinebox
}%
\fi
\else
\hbox to \myboxwidth{%
$\Pi_{14,6}$ spans $L_{30.27}$%
\hfil}%
\hbox to \myboxwidth{%
$\infty|\infty22\slashinfty22\infty|\infty22\slashinfty22\rtimes D_{4}$%
\hfil}%
\box\matricesbox
\fi
}%
\hfill\discretionary{}{}{}%
\setbox\matricesbox=\hbox{%
{$\left[\!\llap{\phantom{%
\begingroup \smaller\smaller\smaller\begin{tabular}{@{}c@{}}%
\phantom{0}\\\phantom{0}\\\phantom{0}
\end{tabular}\endgroup%
}}\right.$}%
\begingroup \smaller\smaller\smaller\begin{tabular}{@{}c@{}}%
-1\\\phantom{0}\\\phantom{0}
\end{tabular}\endgroup%
\kern3pt%
\begingroup \smaller\smaller\smaller\begin{tabular}{@{}c@{}}%
\phantom{0}\\15/2\\\phantom{0}
\end{tabular}\endgroup%
\kern3pt%
\begingroup \smaller\smaller\smaller\begin{tabular}{@{}c@{}}%
\phantom{0}\\\phantom{0}\\45/2
\end{tabular}\endgroup%
{$\left.\llap{\phantom{%
\begingroup \smaller\smaller\smaller\begin{tabular}{@{}c@{}}%
\phantom{0}\\\phantom{0}\\\phantom{0}
\end{tabular}\endgroup%
}}\!\right]$}%
{$\left[\!\llap{\phantom{%
\begingroup \smaller\smaller\smaller\begin{tabular}{@{}c@{}}%
0\\0\\0
\end{tabular}\endgroup%
}}\right.$}%
\begingroup \smaller\smaller\smaller\begin{tabular}{@{}c@{}}%
30\\11\\1
\end{tabular}\endgroup%
\kern3pt%
\begingroup \smaller\smaller\smaller\begin{tabular}{@{}c@{}}%
9\\3\\1
\end{tabular}\endgroup%
\kern3pt%
\begingroup \smaller\smaller\smaller\begin{tabular}{@{}c@{}}%
5\\1\\1
\end{tabular}\endgroup%
\kern3pt%
\begingroup \smaller\smaller\smaller\begin{tabular}{@{}c@{}}%
9\\0\\2
\end{tabular}\endgroup%
{$\left.\llap{\phantom{%
\begingroup \smaller\smaller\smaller\begin{tabular}{@{}c@{}}%
0\\0\\0
\end{tabular}\endgroup%
}}\!\right]$}%
}%
\ifdim\wd\matricesbox>\halfwidth\myboxwidth=\hsize\else\myboxwidth=\halfwidth\fi
\vbox{%
\ifdim\myboxwidth=\hsize
\setbox\onelinebox=\hbox{%
\vbox{\hbox{%
$\Pi_{14,7}$ spans $L_{16.13}$%
}\hbox{%
$\slashthree222|222\slashthree222|222\rtimes D_{4}$%
}%
}%
\hfill\copy\matricesbox
}%
\ifdim\wd\onelinebox>\myboxwidth
\hbox to \myboxwidth{%
$\Pi_{14,7}$ spans $L_{16.13}$%
\hfil
$\slashthree222|222\slashthree222|222\rtimes D_{4}$%
}%
\box\matricesbox
\else
\hbox to \myboxwidth{%
\unhbox\onelinebox
}%
\fi
\else
\hbox to \myboxwidth{%
$\Pi_{14,7}$ spans $L_{16.13}$%
\hfil}%
\hbox to \myboxwidth{%
$\slashthree222|222\slashthree222|222\rtimes D_{4}$%
\hfil}%
\box\matricesbox
\fi
}%
\hfill\discretionary{}{}{}%
\setbox\matricesbox=\hbox{%
{$\left[\!\llap{\phantom{%
\begingroup \smaller\smaller\smaller\begin{tabular}{@{}c@{}}%
\phantom{0}\\\phantom{0}\\\phantom{0}
\end{tabular}\endgroup%
}}\right.$}%
\begingroup \smaller\smaller\smaller\begin{tabular}{@{}c@{}}%
-8\\\phantom{0}\\\phantom{0}
\end{tabular}\endgroup%
\kern3pt%
\begingroup \smaller\smaller\smaller\begin{tabular}{@{}c@{}}%
\phantom{0}\\1\\\phantom{0}
\end{tabular}\endgroup%
\kern3pt%
\begingroup \smaller\smaller\smaller\begin{tabular}{@{}c@{}}%
\phantom{0}\\\phantom{0}\\2
\end{tabular}\endgroup%
{$\left.\llap{\phantom{%
\begingroup \smaller\smaller\smaller\begin{tabular}{@{}c@{}}%
\phantom{0}\\\phantom{0}\\\phantom{0}
\end{tabular}\endgroup%
}}\!\right]$}%
{$\left[\!\llap{\phantom{%
\begingroup \smaller\smaller\smaller\begin{tabular}{@{}c@{}}%
0\\0\\0
\end{tabular}\endgroup%
}}\right.$}%
\begingroup \smaller\smaller\smaller\begin{tabular}{@{}c@{}}%
1\\-3\\0
\end{tabular}\endgroup%
\kern3pt%
\begingroup \smaller\smaller\smaller\begin{tabular}{@{}c@{}}%
4\\-10\\-4
\end{tabular}\endgroup%
\kern3pt%
\begingroup \smaller\smaller\smaller\begin{tabular}{@{}c@{}}%
2\\-4\\-3
\end{tabular}\endgroup%
\kern3pt%
\begingroup \smaller\smaller\smaller\begin{tabular}{@{}c@{}}%
1\\-1\\-2
\end{tabular}\endgroup%
{$\left.\llap{\phantom{%
\begingroup \smaller\smaller\smaller\begin{tabular}{@{}c@{}}%
0\\0\\0
\end{tabular}\endgroup%
}}\!\right]$}%
}%
\ifdim\wd\matricesbox>\halfwidth\myboxwidth=\hsize\else\myboxwidth=\halfwidth\fi
\vbox{%
\ifdim\myboxwidth=\hsize
\setbox\onelinebox=\hbox{%
\vbox{\hbox{%
$\Pi_{14,8}$ spans $L_{161.6}$%
}\hbox{%
$|\infty22\slashinfty22\infty|\infty22\slashinfty22\infty\rtimes D_{4}$%
}%
}%
\hfill\copy\matricesbox
}%
\ifdim\wd\onelinebox>\myboxwidth
\hbox to \myboxwidth{%
$\Pi_{14,8}$ spans $L_{161.6}$%
\hfil
$|\infty22\slashinfty22\infty|\infty22\slashinfty22\infty\rtimes D_{4}$%
}%
\box\matricesbox
\else
\hbox to \myboxwidth{%
\unhbox\onelinebox
}%
\fi
\else
\hbox to \myboxwidth{%
$\Pi_{14,8}$ spans $L_{161.6}$%
\hfil}%
\hbox to \myboxwidth{%
$|\infty22\slashinfty22\infty|\infty22\slashinfty22\infty\rtimes D_{4}$%
\hfil}%
\box\matricesbox
\fi
}%
\hfill\discretionary{}{}{}%
\setbox\matricesbox=\hbox{%
{$\left[\!\llap{\phantom{%
\begingroup \smaller\smaller\smaller\begin{tabular}{@{}c@{}}%
\phantom{0}\\\phantom{0}\\\phantom{0}
\end{tabular}\endgroup%
}}\right.$}%
\begingroup \smaller\smaller\smaller\begin{tabular}{@{}c@{}}%
-1\\\phantom{0}\\\phantom{0}
\end{tabular}\endgroup%
\kern3pt%
\begingroup \smaller\smaller\smaller\begin{tabular}{@{}c@{}}%
\phantom{0}\\15\\\phantom{0}
\end{tabular}\endgroup%
\kern3pt%
\begingroup \smaller\smaller\smaller\begin{tabular}{@{}c@{}}%
\phantom{0}\\\phantom{0}\\5
\end{tabular}\endgroup%
{$\left.\llap{\phantom{%
\begingroup \smaller\smaller\smaller\begin{tabular}{@{}c@{}}%
\phantom{0}\\\phantom{0}\\\phantom{0}
\end{tabular}\endgroup%
}}\!\right]$}%
{$\left[\!\llap{\phantom{%
\begingroup \smaller\smaller\smaller\begin{tabular}{@{}c@{}}%
0\\0\\0
\end{tabular}\endgroup%
}}\right.$}%
\begingroup \smaller\smaller\smaller\begin{tabular}{@{}c@{}}%
15\\-4\\0
\end{tabular}\endgroup%
\kern3pt%
\begingroup \smaller\smaller\smaller\begin{tabular}{@{}c@{}}%
4\\-1\\1
\end{tabular}\endgroup%
\kern3pt%
\begingroup \smaller\smaller\smaller\begin{tabular}{@{}c@{}}%
15\\-2\\6
\end{tabular}\endgroup%
\kern3pt%
\begingroup \smaller\smaller\smaller\begin{tabular}{@{}c@{}}%
4\\0\\2
\end{tabular}\endgroup%
\kern3pt%
\begingroup \smaller\smaller\smaller\begin{tabular}{@{}c@{}}%
15\\2\\6
\end{tabular}\endgroup%
\kern3pt%
\begingroup \smaller\smaller\smaller\begin{tabular}{@{}c@{}}%
20\\4\\6
\end{tabular}\endgroup%
\kern3pt%
\begingroup \smaller\smaller\smaller\begin{tabular}{@{}c@{}}%
20\\5\\3
\end{tabular}\endgroup%
\kern3pt%
\begingroup \smaller\smaller\smaller\begin{tabular}{@{}c@{}}%
15\\4\\0
\end{tabular}\endgroup%
{$\left.\llap{\phantom{%
\begingroup \smaller\smaller\smaller\begin{tabular}{@{}c@{}}%
0\\0\\0
\end{tabular}\endgroup%
}}\!\right]$}%
}%
\ifdim\wd\matricesbox>\halfwidth\myboxwidth=\hsize\else\myboxwidth=\halfwidth\fi
\vbox{%
\ifdim\myboxwidth=\hsize
\setbox\onelinebox=\hbox{%
\vbox{\hbox{%
$\Pi_{14,9}$ spans $L_{31.7}$%
}\hbox{%
$222|2222232|2322\rtimes D_{2}$%
}%
}%
\hfill\copy\matricesbox
}%
\ifdim\wd\onelinebox>\myboxwidth
\hbox to \myboxwidth{%
$\Pi_{14,9}$ spans $L_{31.7}$%
\hfil
$222|2222232|2322\rtimes D_{2}$%
}%
\box\matricesbox
\else
\hbox to \myboxwidth{%
\unhbox\onelinebox
}%
\fi
\else
\hbox to \myboxwidth{%
$\Pi_{14,9}$ spans $L_{31.7}$%
\hfil}%
\hbox to \myboxwidth{%
$222|2222232|2322\rtimes D_{2}$%
\hfil}%
\box\matricesbox
\fi
}%
\hfill\discretionary{}{}{}%
\setbox\matricesbox=\hbox{%
{$\left[\!\llap{\phantom{%
\begingroup \smaller\smaller\smaller\begin{tabular}{@{}c@{}}%
\phantom{0}\\\phantom{0}\\\phantom{0}
\end{tabular}\endgroup%
}}\right.$}%
\begingroup \smaller\smaller\smaller\begin{tabular}{@{}c@{}}%
-1\\\phantom{0}\\\phantom{0}
\end{tabular}\endgroup%
\kern3pt%
\begingroup \smaller\smaller\smaller\begin{tabular}{@{}c@{}}%
\phantom{0}\\5\\\phantom{0}
\end{tabular}\endgroup%
\kern3pt%
\begingroup \smaller\smaller\smaller\begin{tabular}{@{}c@{}}%
\phantom{0}\\\phantom{0}\\15
\end{tabular}\endgroup%
{$\left.\llap{\phantom{%
\begingroup \smaller\smaller\smaller\begin{tabular}{@{}c@{}}%
\phantom{0}\\\phantom{0}\\\phantom{0}
\end{tabular}\endgroup%
}}\!\right]$}%
{$\left[\!\llap{\phantom{%
\begingroup \smaller\smaller\smaller\begin{tabular}{@{}c@{}}%
0\\0\\0
\end{tabular}\endgroup%
}}\right.$}%
\begingroup \smaller\smaller\smaller\begin{tabular}{@{}c@{}}%
4\\2\\0
\end{tabular}\endgroup%
\kern3pt%
\begingroup \smaller\smaller\smaller\begin{tabular}{@{}c@{}}%
15\\6\\2
\end{tabular}\endgroup%
\kern3pt%
\begingroup \smaller\smaller\smaller\begin{tabular}{@{}c@{}}%
4\\1\\1
\end{tabular}\endgroup%
\kern3pt%
\begingroup \smaller\smaller\smaller\begin{tabular}{@{}c@{}}%
15\\0\\4
\end{tabular}\endgroup%
\kern3pt%
\begingroup \smaller\smaller\smaller\begin{tabular}{@{}c@{}}%
20\\-3\\5
\end{tabular}\endgroup%
\kern3pt%
\begingroup \smaller\smaller\smaller\begin{tabular}{@{}c@{}}%
20\\-6\\4
\end{tabular}\endgroup%
\kern3pt%
\begingroup \smaller\smaller\smaller\begin{tabular}{@{}c@{}}%
15\\-6\\2
\end{tabular}\endgroup%
\kern3pt%
\begingroup \smaller\smaller\smaller\begin{tabular}{@{}c@{}}%
4\\-2\\0
\end{tabular}\endgroup%
{$\left.\llap{\phantom{%
\begingroup \smaller\smaller\smaller\begin{tabular}{@{}c@{}}%
0\\0\\0
\end{tabular}\endgroup%
}}\!\right]$}%
}%
\ifdim\wd\matricesbox>\halfwidth\myboxwidth=\hsize\else\myboxwidth=\halfwidth\fi
\vbox{%
\ifdim\myboxwidth=\hsize
\setbox\onelinebox=\hbox{%
\vbox{\hbox{%
$\Pi_{14,10}$ spans $L_{31.7}$%
}\hbox{%
$22|2222322|22322\rtimes D_{2}$%
}%
}%
\hfill\copy\matricesbox
}%
\ifdim\wd\onelinebox>\myboxwidth
\hbox to \myboxwidth{%
$\Pi_{14,10}$ spans $L_{31.7}$%
\hfil
$22|2222322|22322\rtimes D_{2}$%
}%
\box\matricesbox
\else
\hbox to \myboxwidth{%
\unhbox\onelinebox
}%
\fi
\else
\hbox to \myboxwidth{%
$\Pi_{14,10}$ spans $L_{31.7}$%
\hfil}%
\hbox to \myboxwidth{%
$22|2222322|22322\rtimes D_{2}$%
\hfil}%
\box\matricesbox
\fi
}%
\hfill\discretionary{}{}{}%
\setbox\matricesbox=\hbox{%
{$\left[\!\llap{\phantom{%
\begingroup \smaller\smaller\smaller\begin{tabular}{@{}c@{}}%
\phantom{0}\\\phantom{0}\\\phantom{0}
\end{tabular}\endgroup%
}}\right.$}%
\begingroup \smaller\smaller\smaller\begin{tabular}{@{}c@{}}%
-1\\\phantom{0}\\\phantom{0}
\end{tabular}\endgroup%
\kern3pt%
\begingroup \smaller\smaller\smaller\begin{tabular}{@{}c@{}}%
\phantom{0}\\15/2\\\phantom{0}
\end{tabular}\endgroup%
\kern3pt%
\begingroup \smaller\smaller\smaller\begin{tabular}{@{}c@{}}%
\phantom{0}\\\phantom{0}\\45/2
\end{tabular}\endgroup%
{$\left.\llap{\phantom{%
\begingroup \smaller\smaller\smaller\begin{tabular}{@{}c@{}}%
\phantom{0}\\\phantom{0}\\\phantom{0}
\end{tabular}\endgroup%
}}\!\right]$}%
{$\left[\!\llap{\phantom{%
\begingroup \smaller\smaller\smaller\begin{tabular}{@{}c@{}}%
0\\0\\0
\end{tabular}\endgroup%
}}\right.$}%
\begingroup \smaller\smaller\smaller\begin{tabular}{@{}c@{}}%
5\\-2\\0
\end{tabular}\endgroup%
\kern3pt%
\begingroup \smaller\smaller\smaller\begin{tabular}{@{}c@{}}%
9\\-3\\1
\end{tabular}\endgroup%
\kern3pt%
\begingroup \smaller\smaller\smaller\begin{tabular}{@{}c@{}}%
5\\-1\\1
\end{tabular}\endgroup%
\kern3pt%
\begingroup \smaller\smaller\smaller\begin{tabular}{@{}c@{}}%
9\\0\\2
\end{tabular}\endgroup%
\kern3pt%
\begingroup \smaller\smaller\smaller\begin{tabular}{@{}c@{}}%
5\\1\\1
\end{tabular}\endgroup%
\kern3pt%
\begingroup \smaller\smaller\smaller\begin{tabular}{@{}c@{}}%
90\\27\\11
\end{tabular}\endgroup%
\kern3pt%
\begingroup \smaller\smaller\smaller\begin{tabular}{@{}c@{}}%
90\\30\\8
\end{tabular}\endgroup%
\kern3pt%
\begingroup \smaller\smaller\smaller\begin{tabular}{@{}c@{}}%
5\\2\\0
\end{tabular}\endgroup%
{$\left.\llap{\phantom{%
\begingroup \smaller\smaller\smaller\begin{tabular}{@{}c@{}}%
0\\0\\0
\end{tabular}\endgroup%
}}\!\right]$}%
}%
\ifdim\wd\matricesbox>\halfwidth\myboxwidth=\hsize\else\myboxwidth=\halfwidth\fi
\vbox{%
\ifdim\myboxwidth=\hsize
\setbox\onelinebox=\hbox{%
\vbox{\hbox{%
$\Pi_{14,11}$ spans $L_{16.13}$%
}\hbox{%
$2222|2222232|232\rtimes D_{2}$%
}%
}%
\hfill\copy\matricesbox
}%
\ifdim\wd\onelinebox>\myboxwidth
\hbox to \myboxwidth{%
$\Pi_{14,11}$ spans $L_{16.13}$%
\hfil
$2222|2222232|232\rtimes D_{2}$%
}%
\box\matricesbox
\else
\hbox to \myboxwidth{%
\unhbox\onelinebox
}%
\fi
\else
\hbox to \myboxwidth{%
$\Pi_{14,11}$ spans $L_{16.13}$%
\hfil}%
\hbox to \myboxwidth{%
$2222|2222232|232\rtimes D_{2}$%
\hfil}%
\box\matricesbox
\fi
}%
\hfill\discretionary{}{}{}%
\setbox\matricesbox=\hbox{%
{$\left[\!\llap{\phantom{%
\begingroup \smaller\smaller\smaller\begin{tabular}{@{}c@{}}%
\phantom{0}\\\phantom{0}\\\phantom{0}
\end{tabular}\endgroup%
}}\right.$}%
\begingroup \smaller\smaller\smaller\begin{tabular}{@{}c@{}}%
-1\\\phantom{0}\\\phantom{0}
\end{tabular}\endgroup%
\kern3pt%
\begingroup \smaller\smaller\smaller\begin{tabular}{@{}c@{}}%
\phantom{0}\\45/2\\\phantom{0}
\end{tabular}\endgroup%
\kern3pt%
\begingroup \smaller\smaller\smaller\begin{tabular}{@{}c@{}}%
\phantom{0}\\\phantom{0}\\15/2
\end{tabular}\endgroup%
{$\left.\llap{\phantom{%
\begingroup \smaller\smaller\smaller\begin{tabular}{@{}c@{}}%
\phantom{0}\\\phantom{0}\\\phantom{0}
\end{tabular}\endgroup%
}}\!\right]$}%
{$\left[\!\llap{\phantom{%
\begingroup \smaller\smaller\smaller\begin{tabular}{@{}c@{}}%
0\\0\\0
\end{tabular}\endgroup%
}}\right.$}%
\begingroup \smaller\smaller\smaller\begin{tabular}{@{}c@{}}%
9\\2\\0
\end{tabular}\endgroup%
\kern3pt%
\begingroup \smaller\smaller\smaller\begin{tabular}{@{}c@{}}%
5\\1\\1
\end{tabular}\endgroup%
\kern3pt%
\begingroup \smaller\smaller\smaller\begin{tabular}{@{}c@{}}%
9\\1\\3
\end{tabular}\endgroup%
\kern3pt%
\begingroup \smaller\smaller\smaller\begin{tabular}{@{}c@{}}%
5\\0\\2
\end{tabular}\endgroup%
\kern3pt%
\begingroup \smaller\smaller\smaller\begin{tabular}{@{}c@{}}%
9\\-1\\3
\end{tabular}\endgroup%
\kern3pt%
\begingroup \smaller\smaller\smaller\begin{tabular}{@{}c@{}}%
30\\-5\\7
\end{tabular}\endgroup%
\kern3pt%
\begingroup \smaller\smaller\smaller\begin{tabular}{@{}c@{}}%
30\\-6\\4
\end{tabular}\endgroup%
\kern3pt%
\begingroup \smaller\smaller\smaller\begin{tabular}{@{}c@{}}%
9\\-2\\0
\end{tabular}\endgroup%
{$\left.\llap{\phantom{%
\begingroup \smaller\smaller\smaller\begin{tabular}{@{}c@{}}%
0\\0\\0
\end{tabular}\endgroup%
}}\!\right]$}%
}%
\ifdim\wd\matricesbox>\halfwidth\myboxwidth=\hsize\else\myboxwidth=\halfwidth\fi
\vbox{%
\ifdim\myboxwidth=\hsize
\setbox\onelinebox=\hbox{%
\vbox{\hbox{%
$\Pi_{14,12}$ spans $L_{16.13}$%
}\hbox{%
$222|2222232|2322\rtimes D_{2}$%
}%
}%
\hfill\copy\matricesbox
}%
\ifdim\wd\onelinebox>\myboxwidth
\hbox to \myboxwidth{%
$\Pi_{14,12}$ spans $L_{16.13}$%
\hfil
$222|2222232|2322\rtimes D_{2}$%
}%
\box\matricesbox
\else
\hbox to \myboxwidth{%
\unhbox\onelinebox
}%
\fi
\else
\hbox to \myboxwidth{%
$\Pi_{14,12}$ spans $L_{16.13}$%
\hfil}%
\hbox to \myboxwidth{%
$222|2222232|2322\rtimes D_{2}$%
\hfil}%
\box\matricesbox
\fi
}%
\hfill\discretionary{}{}{}%
\setbox\matricesbox=\hbox{%
{$\left[\!\llap{\phantom{%
\begingroup \smaller\smaller\smaller\begin{tabular}{@{}c@{}}%
\phantom{0}\\\phantom{0}\\\phantom{0}
\end{tabular}\endgroup%
}}\right.$}%
\begingroup \smaller\smaller\smaller\begin{tabular}{@{}c@{}}%
-1\\\phantom{0}\\\phantom{0}
\end{tabular}\endgroup%
\kern3pt%
\begingroup \smaller\smaller\smaller\begin{tabular}{@{}c@{}}%
\phantom{0}\\45/2\\\phantom{0}
\end{tabular}\endgroup%
\kern3pt%
\begingroup \smaller\smaller\smaller\begin{tabular}{@{}c@{}}%
\phantom{0}\\\phantom{0}\\15/2
\end{tabular}\endgroup%
{$\left.\llap{\phantom{%
\begingroup \smaller\smaller\smaller\begin{tabular}{@{}c@{}}%
\phantom{0}\\\phantom{0}\\\phantom{0}
\end{tabular}\endgroup%
}}\!\right]$}%
{$\left[\!\llap{\phantom{%
\begingroup \smaller\smaller\smaller\begin{tabular}{@{}c@{}}%
0\\0\\0
\end{tabular}\endgroup%
}}\right.$}%
\begingroup \smaller\smaller\smaller\begin{tabular}{@{}c@{}}%
9\\2\\0
\end{tabular}\endgroup%
\kern3pt%
\begingroup \smaller\smaller\smaller\begin{tabular}{@{}c@{}}%
5\\1\\1
\end{tabular}\endgroup%
\kern3pt%
\begingroup \smaller\smaller\smaller\begin{tabular}{@{}c@{}}%
9\\1\\3
\end{tabular}\endgroup%
\kern3pt%
\begingroup \smaller\smaller\smaller\begin{tabular}{@{}c@{}}%
5\\0\\2
\end{tabular}\endgroup%
\kern3pt%
\begingroup \smaller\smaller\smaller\begin{tabular}{@{}c@{}}%
90\\-8\\30
\end{tabular}\endgroup%
\kern3pt%
\begingroup \smaller\smaller\smaller\begin{tabular}{@{}c@{}}%
90\\-11\\27
\end{tabular}\endgroup%
\kern3pt%
\begingroup \smaller\smaller\smaller\begin{tabular}{@{}c@{}}%
5\\-1\\1
\end{tabular}\endgroup%
\kern3pt%
\begingroup \smaller\smaller\smaller\begin{tabular}{@{}c@{}}%
9\\-2\\0
\end{tabular}\endgroup%
{$\left.\llap{\phantom{%
\begingroup \smaller\smaller\smaller\begin{tabular}{@{}c@{}}%
0\\0\\0
\end{tabular}\endgroup%
}}\!\right]$}%
}%
\ifdim\wd\matricesbox>\halfwidth\myboxwidth=\hsize\else\myboxwidth=\halfwidth\fi
\vbox{%
\ifdim\myboxwidth=\hsize
\setbox\onelinebox=\hbox{%
\vbox{\hbox{%
$\Pi_{14,13}$ spans $L_{16.13}$%
}\hbox{%
$222|2222322|2232\rtimes D_{2}$%
}%
}%
\hfill\copy\matricesbox
}%
\ifdim\wd\onelinebox>\myboxwidth
\hbox to \myboxwidth{%
$\Pi_{14,13}$ spans $L_{16.13}$%
\hfil
$222|2222322|2232\rtimes D_{2}$%
}%
\box\matricesbox
\else
\hbox to \myboxwidth{%
\unhbox\onelinebox
}%
\fi
\else
\hbox to \myboxwidth{%
$\Pi_{14,13}$ spans $L_{16.13}$%
\hfil}%
\hbox to \myboxwidth{%
$222|2222322|2232\rtimes D_{2}$%
\hfil}%
\box\matricesbox
\fi
}%
\hfill\discretionary{}{}{}%
\setbox\matricesbox=\hbox{%
{$\left[\!\llap{\phantom{%
\begingroup \smaller\smaller\smaller\begin{tabular}{@{}c@{}}%
\phantom{0}\\\phantom{0}\\\phantom{0}
\end{tabular}\endgroup%
}}\right.$}%
\begingroup \smaller\smaller\smaller\begin{tabular}{@{}c@{}}%
-1\\\phantom{0}\\\phantom{0}
\end{tabular}\endgroup%
\kern3pt%
\begingroup \smaller\smaller\smaller\begin{tabular}{@{}c@{}}%
\phantom{0}\\15/2\\\phantom{0}
\end{tabular}\endgroup%
\kern3pt%
\begingroup \smaller\smaller\smaller\begin{tabular}{@{}c@{}}%
\phantom{0}\\\phantom{0}\\45/2
\end{tabular}\endgroup%
{$\left.\llap{\phantom{%
\begingroup \smaller\smaller\smaller\begin{tabular}{@{}c@{}}%
\phantom{0}\\\phantom{0}\\\phantom{0}
\end{tabular}\endgroup%
}}\!\right]$}%
{$\left[\!\llap{\phantom{%
\begingroup \smaller\smaller\smaller\begin{tabular}{@{}c@{}}%
0\\0\\0
\end{tabular}\endgroup%
}}\right.$}%
\begingroup \smaller\smaller\smaller\begin{tabular}{@{}c@{}}%
5\\-2\\0
\end{tabular}\endgroup%
\kern3pt%
\begingroup \smaller\smaller\smaller\begin{tabular}{@{}c@{}}%
9\\-3\\1
\end{tabular}\endgroup%
\kern3pt%
\begingroup \smaller\smaller\smaller\begin{tabular}{@{}c@{}}%
5\\-1\\1
\end{tabular}\endgroup%
\kern3pt%
\begingroup \smaller\smaller\smaller\begin{tabular}{@{}c@{}}%
9\\0\\2
\end{tabular}\endgroup%
\kern3pt%
\begingroup \smaller\smaller\smaller\begin{tabular}{@{}c@{}}%
30\\4\\6
\end{tabular}\endgroup%
\kern3pt%
\begingroup \smaller\smaller\smaller\begin{tabular}{@{}c@{}}%
30\\7\\5
\end{tabular}\endgroup%
\kern3pt%
\begingroup \smaller\smaller\smaller\begin{tabular}{@{}c@{}}%
9\\3\\1
\end{tabular}\endgroup%
\kern3pt%
\begingroup \smaller\smaller\smaller\begin{tabular}{@{}c@{}}%
5\\2\\0
\end{tabular}\endgroup%
{$\left.\llap{\phantom{%
\begingroup \smaller\smaller\smaller\begin{tabular}{@{}c@{}}%
0\\0\\0
\end{tabular}\endgroup%
}}\!\right]$}%
}%
\ifdim\wd\matricesbox>\halfwidth\myboxwidth=\hsize\else\myboxwidth=\halfwidth\fi
\vbox{%
\ifdim\myboxwidth=\hsize
\setbox\onelinebox=\hbox{%
\vbox{\hbox{%
$\Pi_{14,14}$ spans $L_{16.13}$%
}\hbox{%
$22|2222322|22322\rtimes D_{2}$%
}%
}%
\hfill\copy\matricesbox
}%
\ifdim\wd\onelinebox>\myboxwidth
\hbox to \myboxwidth{%
$\Pi_{14,14}$ spans $L_{16.13}$%
\hfil
$22|2222322|22322\rtimes D_{2}$%
}%
\box\matricesbox
\else
\hbox to \myboxwidth{%
\unhbox\onelinebox
}%
\fi
\else
\hbox to \myboxwidth{%
$\Pi_{14,14}$ spans $L_{16.13}$%
\hfil}%
\hbox to \myboxwidth{%
$22|2222322|22322\rtimes D_{2}$%
\hfil}%
\box\matricesbox
\fi
}%
\hfill\discretionary{}{}{}%
\setbox\matricesbox=\hbox{%
{$\left[\!\llap{\phantom{%
\begingroup \smaller\smaller\smaller\begin{tabular}{@{}c@{}}%
\phantom{0}\\\phantom{0}\\\phantom{0}
\end{tabular}\endgroup%
}}\right.$}%
\begingroup \smaller\smaller\smaller\begin{tabular}{@{}c@{}}%
-1\\\phantom{0}\\\phantom{0}
\end{tabular}\endgroup%
\kern3pt%
\begingroup \smaller\smaller\smaller\begin{tabular}{@{}c@{}}%
\phantom{0}\\7/2\\\phantom{0}
\end{tabular}\endgroup%
\kern3pt%
\begingroup \smaller\smaller\smaller\begin{tabular}{@{}c@{}}%
\phantom{0}\\\phantom{0}\\21/2
\end{tabular}\endgroup%
{$\left.\llap{\phantom{%
\begingroup \smaller\smaller\smaller\begin{tabular}{@{}c@{}}%
\phantom{0}\\\phantom{0}\\\phantom{0}
\end{tabular}\endgroup%
}}\!\right]$}%
{$\left[\!\llap{\phantom{%
\begingroup \smaller\smaller\smaller\begin{tabular}{@{}c@{}}%
0\\0\\0
\end{tabular}\endgroup%
}}\right.$}%
\begingroup \smaller\smaller\smaller\begin{tabular}{@{}c@{}}%
7\\4\\0
\end{tabular}\endgroup%
\kern3pt%
\begingroup \smaller\smaller\smaller\begin{tabular}{@{}c@{}}%
6\\3\\1
\end{tabular}\endgroup%
\kern3pt%
\begingroup \smaller\smaller\smaller\begin{tabular}{@{}c@{}}%
7\\2\\2
\end{tabular}\endgroup%
\kern3pt%
\begingroup \smaller\smaller\smaller\begin{tabular}{@{}c@{}}%
6\\0\\2
\end{tabular}\endgroup%
\kern3pt%
\begingroup \smaller\smaller\smaller\begin{tabular}{@{}c@{}}%
7\\-2\\2
\end{tabular}\endgroup%
\kern3pt%
\begingroup \smaller\smaller\smaller\begin{tabular}{@{}c@{}}%
42\\-18\\8
\end{tabular}\endgroup%
\kern3pt%
\begingroup \smaller\smaller\smaller\begin{tabular}{@{}c@{}}%
42\\-21\\5
\end{tabular}\endgroup%
\kern3pt%
\begingroup \smaller\smaller\smaller\begin{tabular}{@{}c@{}}%
7\\-4\\0
\end{tabular}\endgroup%
{$\left.\llap{\phantom{%
\begingroup \smaller\smaller\smaller\begin{tabular}{@{}c@{}}%
0\\0\\0
\end{tabular}\endgroup%
}}\!\right]$}%
}%
\ifdim\wd\matricesbox>\halfwidth\myboxwidth=\hsize\else\myboxwidth=\halfwidth\fi
\vbox{%
\ifdim\myboxwidth=\hsize
\setbox\onelinebox=\hbox{%
\vbox{\hbox{%
$\Pi_{14,15}$ spans $L_{22.4}$%
}\hbox{%
$222|2222232|2322\rtimes D_{2}$%
}%
}%
\hfill\copy\matricesbox
}%
\ifdim\wd\onelinebox>\myboxwidth
\hbox to \myboxwidth{%
$\Pi_{14,15}$ spans $L_{22.4}$%
\hfil
$222|2222232|2322\rtimes D_{2}$%
}%
\box\matricesbox
\else
\hbox to \myboxwidth{%
\unhbox\onelinebox
}%
\fi
\else
\hbox to \myboxwidth{%
$\Pi_{14,15}$ spans $L_{22.4}$%
\hfil}%
\hbox to \myboxwidth{%
$222|2222232|2322\rtimes D_{2}$%
\hfil}%
\box\matricesbox
\fi
}%
\hfill\discretionary{}{}{}%
\setbox\matricesbox=\hbox{%
{$\left[\!\llap{\phantom{%
\begingroup \smaller\smaller\smaller\begin{tabular}{@{}c@{}}%
\phantom{0}\\\phantom{0}\\\phantom{0}
\end{tabular}\endgroup%
}}\right.$}%
\begingroup \smaller\smaller\smaller\begin{tabular}{@{}c@{}}%
-1\\\phantom{0}\\\phantom{0}
\end{tabular}\endgroup%
\kern3pt%
\begingroup \smaller\smaller\smaller\begin{tabular}{@{}c@{}}%
\phantom{0}\\21/2\\\phantom{0}
\end{tabular}\endgroup%
\kern3pt%
\begingroup \smaller\smaller\smaller\begin{tabular}{@{}c@{}}%
\phantom{0}\\\phantom{0}\\7/2
\end{tabular}\endgroup%
{$\left.\llap{\phantom{%
\begingroup \smaller\smaller\smaller\begin{tabular}{@{}c@{}}%
\phantom{0}\\\phantom{0}\\\phantom{0}
\end{tabular}\endgroup%
}}\!\right]$}%
{$\left[\!\llap{\phantom{%
\begingroup \smaller\smaller\smaller\begin{tabular}{@{}c@{}}%
0\\0\\0
\end{tabular}\endgroup%
}}\right.$}%
\begingroup \smaller\smaller\smaller\begin{tabular}{@{}c@{}}%
6\\2\\0
\end{tabular}\endgroup%
\kern3pt%
\begingroup \smaller\smaller\smaller\begin{tabular}{@{}c@{}}%
7\\2\\2
\end{tabular}\endgroup%
\kern3pt%
\begingroup \smaller\smaller\smaller\begin{tabular}{@{}c@{}}%
6\\1\\3
\end{tabular}\endgroup%
\kern3pt%
\begingroup \smaller\smaller\smaller\begin{tabular}{@{}c@{}}%
7\\0\\4
\end{tabular}\endgroup%
\kern3pt%
\begingroup \smaller\smaller\smaller\begin{tabular}{@{}c@{}}%
42\\-5\\21
\end{tabular}\endgroup%
\kern3pt%
\begingroup \smaller\smaller\smaller\begin{tabular}{@{}c@{}}%
42\\-8\\18
\end{tabular}\endgroup%
\kern3pt%
\begingroup \smaller\smaller\smaller\begin{tabular}{@{}c@{}}%
7\\-2\\2
\end{tabular}\endgroup%
\kern3pt%
\begingroup \smaller\smaller\smaller\begin{tabular}{@{}c@{}}%
6\\-2\\0
\end{tabular}\endgroup%
{$\left.\llap{\phantom{%
\begingroup \smaller\smaller\smaller\begin{tabular}{@{}c@{}}%
0\\0\\0
\end{tabular}\endgroup%
}}\!\right]$}%
}%
\ifdim\wd\matricesbox>\halfwidth\myboxwidth=\hsize\else\myboxwidth=\halfwidth\fi
\vbox{%
\ifdim\myboxwidth=\hsize
\setbox\onelinebox=\hbox{%
\vbox{\hbox{%
$\Pi_{14,16}$ spans $L_{22.4}$%
}\hbox{%
$22|2222322|22322\rtimes D_{2}$%
}%
}%
\hfill\copy\matricesbox
}%
\ifdim\wd\onelinebox>\myboxwidth
\hbox to \myboxwidth{%
$\Pi_{14,16}$ spans $L_{22.4}$%
\hfil
$22|2222322|22322\rtimes D_{2}$%
}%
\box\matricesbox
\else
\hbox to \myboxwidth{%
\unhbox\onelinebox
}%
\fi
\else
\hbox to \myboxwidth{%
$\Pi_{14,16}$ spans $L_{22.4}$%
\hfil}%
\hbox to \myboxwidth{%
$22|2222322|22322\rtimes D_{2}$%
\hfil}%
\box\matricesbox
\fi
}%
\hfill\discretionary{}{}{}%
\setbox\matricesbox=\hbox{%
{$\left[\!\llap{\phantom{%
\begingroup \smaller\smaller\smaller\begin{tabular}{@{}c@{}}%
\phantom{0}\\\phantom{0}\\\phantom{0}
\end{tabular}\endgroup%
}}\right.$}%
\begingroup \smaller\smaller\smaller\begin{tabular}{@{}c@{}}%
-2\\\phantom{0}\\\phantom{0}
\end{tabular}\endgroup%
\kern3pt%
\begingroup \smaller\smaller\smaller\begin{tabular}{@{}c@{}}%
\phantom{0}\\3/2\\\phantom{0}
\end{tabular}\endgroup%
\kern3pt%
\begingroup \smaller\smaller\smaller\begin{tabular}{@{}c@{}}%
\phantom{0}\\\phantom{0}\\3/2
\end{tabular}\endgroup%
{$\left.\llap{\phantom{%
\begingroup \smaller\smaller\smaller\begin{tabular}{@{}c@{}}%
\phantom{0}\\\phantom{0}\\\phantom{0}
\end{tabular}\endgroup%
}}\!\right]$}%
{$\left[\!\llap{\phantom{%
\begingroup \smaller\smaller\smaller\begin{tabular}{@{}c@{}}%
0\\0\\0
\end{tabular}\endgroup%
}}\right.$}%
\begingroup \smaller\smaller\smaller\begin{tabular}{@{}c@{}}%
12\\-14\\2
\end{tabular}\endgroup%
\kern3pt%
\begingroup \smaller\smaller\smaller\begin{tabular}{@{}c@{}}%
6\\-6\\4
\end{tabular}\endgroup%
\kern3pt%
\begingroup \smaller\smaller\smaller\begin{tabular}{@{}c@{}}%
6\\-4\\6
\end{tabular}\endgroup%
\kern3pt%
\begingroup \smaller\smaller\smaller\begin{tabular}{@{}c@{}}%
12\\-2\\14
\end{tabular}\endgroup%
\kern3pt%
\begingroup \smaller\smaller\smaller\begin{tabular}{@{}c@{}}%
12\\2\\14
\end{tabular}\endgroup%
\kern3pt%
\begingroup \smaller\smaller\smaller\begin{tabular}{@{}c@{}}%
1\\1\\1
\end{tabular}\endgroup%
\kern3pt%
\begingroup \smaller\smaller\smaller\begin{tabular}{@{}c@{}}%
12\\14\\2
\end{tabular}\endgroup%
{$\left.\llap{\phantom{%
\begingroup \smaller\smaller\smaller\begin{tabular}{@{}c@{}}%
0\\0\\0
\end{tabular}\endgroup%
}}\!\right]$}%
}%
\ifdim\wd\matricesbox>\halfwidth\myboxwidth=\hsize\else\myboxwidth=\halfwidth\fi
\vbox{%
\ifdim\myboxwidth=\hsize
\setbox\onelinebox=\hbox{%
\vbox{\hbox{%
$\Pi_{14,17}$ spans $L_{123.11}$%
}\hbox{%
$22424\slashtwo424222\slashtwo2\rtimes D_{2}$%
}%
}%
\hfill\copy\matricesbox
}%
\ifdim\wd\onelinebox>\myboxwidth
\hbox to \myboxwidth{%
$\Pi_{14,17}$ spans $L_{123.11}$%
\hfil
$22424\slashtwo424222\slashtwo2\rtimes D_{2}$%
}%
\box\matricesbox
\else
\hbox to \myboxwidth{%
\unhbox\onelinebox
}%
\fi
\else
\hbox to \myboxwidth{%
$\Pi_{14,17}$ spans $L_{123.11}$%
\hfil}%
\hbox to \myboxwidth{%
$22424\slashtwo424222\slashtwo2\rtimes D_{2}$%
\hfil}%
\box\matricesbox
\fi
}%
\hfill\discretionary{}{}{}%
\setbox\matricesbox=\hbox{%
{$\left[\!\llap{\phantom{%
\begingroup \smaller\smaller\smaller\begin{tabular}{@{}c@{}}%
\phantom{0}\\\phantom{0}\\\phantom{0}
\end{tabular}\endgroup%
}}\right.$}%
\begingroup \smaller\smaller\smaller\begin{tabular}{@{}c@{}}%
-3\\\phantom{0}\\\phantom{0}
\end{tabular}\endgroup%
\kern3pt%
\begingroup \smaller\smaller\smaller\begin{tabular}{@{}c@{}}%
\phantom{0}\\2\\\phantom{0}
\end{tabular}\endgroup%
\kern3pt%
\begingroup \smaller\smaller\smaller\begin{tabular}{@{}c@{}}%
\phantom{0}\\\phantom{0}\\2
\end{tabular}\endgroup%
{$\left.\llap{\phantom{%
\begingroup \smaller\smaller\smaller\begin{tabular}{@{}c@{}}%
\phantom{0}\\\phantom{0}\\\phantom{0}
\end{tabular}\endgroup%
}}\!\right]$}%
{$\left[\!\llap{\phantom{%
\begingroup \smaller\smaller\smaller\begin{tabular}{@{}c@{}}%
0\\0\\0
\end{tabular}\endgroup%
}}\right.$}%
\begingroup \smaller\smaller\smaller\begin{tabular}{@{}c@{}}%
4\\5\\-1
\end{tabular}\endgroup%
\kern3pt%
\begingroup \smaller\smaller\smaller\begin{tabular}{@{}c@{}}%
8\\8\\-6
\end{tabular}\endgroup%
\kern3pt%
\begingroup \smaller\smaller\smaller\begin{tabular}{@{}c@{}}%
8\\6\\-8
\end{tabular}\endgroup%
\kern3pt%
\begingroup \smaller\smaller\smaller\begin{tabular}{@{}c@{}}%
4\\1\\-5
\end{tabular}\endgroup%
\kern3pt%
\begingroup \smaller\smaller\smaller\begin{tabular}{@{}c@{}}%
4\\-1\\-5
\end{tabular}\endgroup%
\kern3pt%
\begingroup \smaller\smaller\smaller\begin{tabular}{@{}c@{}}%
1\\-1\\-1
\end{tabular}\endgroup%
\kern3pt%
\begingroup \smaller\smaller\smaller\begin{tabular}{@{}c@{}}%
4\\-5\\-1
\end{tabular}\endgroup%
{$\left.\llap{\phantom{%
\begingroup \smaller\smaller\smaller\begin{tabular}{@{}c@{}}%
0\\0\\0
\end{tabular}\endgroup%
}}\!\right]$}%
}%
\ifdim\wd\matricesbox>\halfwidth\myboxwidth=\hsize\else\myboxwidth=\halfwidth\fi
\vbox{%
\ifdim\myboxwidth=\hsize
\setbox\onelinebox=\hbox{%
\vbox{\hbox{%
$\Pi_{14,18}$ spans $L_{123.9}$%
}\hbox{%
$22424\slashtwo424222\slashtwo2\rtimes D_{2}$%
}%
}%
\hfill\copy\matricesbox
}%
\ifdim\wd\onelinebox>\myboxwidth
\hbox to \myboxwidth{%
$\Pi_{14,18}$ spans $L_{123.9}$%
\hfil
$22424\slashtwo424222\slashtwo2\rtimes D_{2}$%
}%
\box\matricesbox
\else
\hbox to \myboxwidth{%
\unhbox\onelinebox
}%
\fi
\else
\hbox to \myboxwidth{%
$\Pi_{14,18}$ spans $L_{123.9}$%
\hfil}%
\hbox to \myboxwidth{%
$22424\slashtwo424222\slashtwo2\rtimes D_{2}$%
\hfil}%
\box\matricesbox
\fi
}%
\hfill\discretionary{}{}{}%
\setbox\matricesbox=\hbox{%
{$\left[\!\llap{\phantom{%
\begingroup \smaller\smaller\smaller\begin{tabular}{@{}c@{}}%
\phantom{0}\\\phantom{0}\\\phantom{0}
\end{tabular}\endgroup%
}}\right.$}%
\begingroup \smaller\smaller\smaller\begin{tabular}{@{}c@{}}%
-1\\\phantom{0}\\\phantom{0}
\end{tabular}\endgroup%
\kern3pt%
\begingroup \smaller\smaller\smaller\begin{tabular}{@{}c@{}}%
\phantom{0}\\18\\\phantom{0}
\end{tabular}\endgroup%
\kern3pt%
\begingroup \smaller\smaller\smaller\begin{tabular}{@{}c@{}}%
\phantom{0}\\\phantom{0}\\18
\end{tabular}\endgroup%
{$\left.\llap{\phantom{%
\begingroup \smaller\smaller\smaller\begin{tabular}{@{}c@{}}%
\phantom{0}\\\phantom{0}\\\phantom{0}
\end{tabular}\endgroup%
}}\!\right]$}%
{$\left[\!\llap{\phantom{%
\begingroup \smaller\smaller\smaller\begin{tabular}{@{}c@{}}%
0\\0\\0
\end{tabular}\endgroup%
}}\right.$}%
\begingroup \smaller\smaller\smaller\begin{tabular}{@{}c@{}}%
8\\2\\0
\end{tabular}\endgroup%
\kern3pt%
\begingroup \smaller\smaller\smaller\begin{tabular}{@{}c@{}}%
9\\2\\1
\end{tabular}\endgroup%
\kern3pt%
\begingroup \smaller\smaller\smaller\begin{tabular}{@{}c@{}}%
9\\1\\2
\end{tabular}\endgroup%
\kern3pt%
\begingroup \smaller\smaller\smaller\begin{tabular}{@{}c@{}}%
8\\0\\2
\end{tabular}\endgroup%
\kern3pt%
\begingroup \smaller\smaller\smaller\begin{tabular}{@{}c@{}}%
72\\-6\\16
\end{tabular}\endgroup%
\kern3pt%
\begingroup \smaller\smaller\smaller\begin{tabular}{@{}c@{}}%
36\\-5\\7
\end{tabular}\endgroup%
\kern3pt%
\begingroup \smaller\smaller\smaller\begin{tabular}{@{}c@{}}%
9\\-2\\1
\end{tabular}\endgroup%
\kern3pt%
\begingroup \smaller\smaller\smaller\begin{tabular}{@{}c@{}}%
8\\-2\\0
\end{tabular}\endgroup%
{$\left.\llap{\phantom{%
\begingroup \smaller\smaller\smaller\begin{tabular}{@{}c@{}}%
0\\0\\0
\end{tabular}\endgroup%
}}\!\right]$}%
}%
\ifdim\wd\matricesbox>\halfwidth\myboxwidth=\hsize\else\myboxwidth=\halfwidth\fi
\vbox{%
\ifdim\myboxwidth=\hsize
\setbox\onelinebox=\hbox{%
\vbox{\hbox{%
$\Pi_{14,19}$ spans $L_{142.20}$%
}\hbox{%
$\infty422\infty2|2\infty224\infty2|2\rtimes D_{2}$%
}%
}%
\hfill\copy\matricesbox
}%
\ifdim\wd\onelinebox>\myboxwidth
\hbox to \myboxwidth{%
$\Pi_{14,19}$ spans $L_{142.20}$%
\hfil
$\infty422\infty2|2\infty224\infty2|2\rtimes D_{2}$%
}%
\box\matricesbox
\else
\hbox to \myboxwidth{%
\unhbox\onelinebox
}%
\fi
\else
\hbox to \myboxwidth{%
$\Pi_{14,19}$ spans $L_{142.20}$%
\hfil}%
\hbox to \myboxwidth{%
$\infty422\infty2|2\infty224\infty2|2\rtimes D_{2}$%
\hfil}%
\box\matricesbox
\fi
}%
\hfill\discretionary{}{}{}%
\setbox\matricesbox=\hbox{%
{$\left[\!\llap{\phantom{%
\begingroup \smaller\smaller\smaller\begin{tabular}{@{}c@{}}%
\phantom{0}\\\phantom{0}\\\phantom{0}
\end{tabular}\endgroup%
}}\right.$}%
\begingroup \smaller\smaller\smaller\begin{tabular}{@{}c@{}}%
-1\\\phantom{0}\\\phantom{0}
\end{tabular}\endgroup%
\kern3pt%
\begingroup \smaller\smaller\smaller\begin{tabular}{@{}c@{}}%
\phantom{0}\\9\\\phantom{0}
\end{tabular}\endgroup%
\kern3pt%
\begingroup \smaller\smaller\smaller\begin{tabular}{@{}c@{}}%
\phantom{0}\\\phantom{0}\\9
\end{tabular}\endgroup%
{$\left.\llap{\phantom{%
\begingroup \smaller\smaller\smaller\begin{tabular}{@{}c@{}}%
\phantom{0}\\\phantom{0}\\\phantom{0}
\end{tabular}\endgroup%
}}\!\right]$}%
{$\left[\!\llap{\phantom{%
\begingroup \smaller\smaller\smaller\begin{tabular}{@{}c@{}}%
0\\0\\0
\end{tabular}\endgroup%
}}\right.$}%
\begingroup \smaller\smaller\smaller\begin{tabular}{@{}c@{}}%
9\\3\\-1
\end{tabular}\endgroup%
\kern3pt%
\begingroup \smaller\smaller\smaller\begin{tabular}{@{}c@{}}%
8\\2\\-2
\end{tabular}\endgroup%
\kern3pt%
\begingroup \smaller\smaller\smaller\begin{tabular}{@{}c@{}}%
72\\10\\-22
\end{tabular}\endgroup%
\kern3pt%
\begingroup \smaller\smaller\smaller\begin{tabular}{@{}c@{}}%
36\\2\\-12
\end{tabular}\endgroup%
\kern3pt%
\begingroup \smaller\smaller\smaller\begin{tabular}{@{}c@{}}%
9\\-1\\-3
\end{tabular}\endgroup%
\kern3pt%
\begingroup \smaller\smaller\smaller\begin{tabular}{@{}c@{}}%
8\\-2\\-2
\end{tabular}\endgroup%
\kern3pt%
\begingroup \smaller\smaller\smaller\begin{tabular}{@{}c@{}}%
9\\-3\\-1
\end{tabular}\endgroup%
{$\left.\llap{\phantom{%
\begingroup \smaller\smaller\smaller\begin{tabular}{@{}c@{}}%
0\\0\\0
\end{tabular}\endgroup%
}}\!\right]$}%
}%
\ifdim\wd\matricesbox>\halfwidth\myboxwidth=\hsize\else\myboxwidth=\halfwidth\fi
\vbox{%
\ifdim\myboxwidth=\hsize
\setbox\onelinebox=\hbox{%
\vbox{\hbox{%
$\Pi_{14,20}$ spans $L_{142.20}$%
}\hbox{%
$\infty422\slashinfty224\infty22\slashinfty22\rtimes D_{2}$%
}%
}%
\hfill\copy\matricesbox
}%
\ifdim\wd\onelinebox>\myboxwidth
\hbox to \myboxwidth{%
$\Pi_{14,20}$ spans $L_{142.20}$%
\hfil
$\infty422\slashinfty224\infty22\slashinfty22\rtimes D_{2}$%
}%
\box\matricesbox
\else
\hbox to \myboxwidth{%
\unhbox\onelinebox
}%
\fi
\else
\hbox to \myboxwidth{%
$\Pi_{14,20}$ spans $L_{142.20}$%
\hfil}%
\hbox to \myboxwidth{%
$\infty422\slashinfty224\infty22\slashinfty22\rtimes D_{2}$%
\hfil}%
\box\matricesbox
\fi
}%
\hfill\discretionary{}{}{}%
\setbox\matricesbox=\hbox{%
{$\left[\!\llap{\phantom{%
\begingroup \smaller\smaller\smaller\begin{tabular}{@{}c@{}}%
\phantom{0}\\\phantom{0}\\\phantom{0}
\end{tabular}\endgroup%
}}\right.$}%
\begingroup \smaller\smaller\smaller\begin{tabular}{@{}c@{}}%
-1\\\phantom{0}\\\phantom{0}
\end{tabular}\endgroup%
\kern3pt%
\begingroup \smaller\smaller\smaller\begin{tabular}{@{}c@{}}%
\phantom{0}\\18\\\phantom{0}
\end{tabular}\endgroup%
\kern3pt%
\begingroup \smaller\smaller\smaller\begin{tabular}{@{}c@{}}%
\phantom{0}\\\phantom{0}\\18
\end{tabular}\endgroup%
{$\left.\llap{\phantom{%
\begingroup \smaller\smaller\smaller\begin{tabular}{@{}c@{}}%
\phantom{0}\\\phantom{0}\\\phantom{0}
\end{tabular}\endgroup%
}}\!\right]$}%
{$\left[\!\llap{\phantom{%
\begingroup \smaller\smaller\smaller\begin{tabular}{@{}c@{}}%
0\\0\\0
\end{tabular}\endgroup%
}}\right.$}%
\begingroup \smaller\smaller\smaller\begin{tabular}{@{}c@{}}%
8\\2\\0
\end{tabular}\endgroup%
\kern3pt%
\begingroup \smaller\smaller\smaller\begin{tabular}{@{}c@{}}%
72\\16\\-6
\end{tabular}\endgroup%
\kern3pt%
\begingroup \smaller\smaller\smaller\begin{tabular}{@{}c@{}}%
36\\7\\-5
\end{tabular}\endgroup%
\kern3pt%
\begingroup \smaller\smaller\smaller\begin{tabular}{@{}c@{}}%
9\\1\\-2
\end{tabular}\endgroup%
\kern3pt%
\begingroup \smaller\smaller\smaller\begin{tabular}{@{}c@{}}%
8\\0\\-2
\end{tabular}\endgroup%
\kern3pt%
\begingroup \smaller\smaller\smaller\begin{tabular}{@{}c@{}}%
9\\-1\\-2
\end{tabular}\endgroup%
\kern3pt%
\begingroup \smaller\smaller\smaller\begin{tabular}{@{}c@{}}%
9\\-2\\-1
\end{tabular}\endgroup%
\kern3pt%
\begingroup \smaller\smaller\smaller\begin{tabular}{@{}c@{}}%
8\\-2\\0
\end{tabular}\endgroup%
{$\left.\llap{\phantom{%
\begingroup \smaller\smaller\smaller\begin{tabular}{@{}c@{}}%
0\\0\\0
\end{tabular}\endgroup%
}}\!\right]$}%
}%
\ifdim\wd\matricesbox>\halfwidth\myboxwidth=\hsize\else\myboxwidth=\halfwidth\fi
\vbox{%
\ifdim\myboxwidth=\hsize
\setbox\onelinebox=\hbox{%
\vbox{\hbox{%
$\Pi_{14,21}$ spans $L_{142.20}$%
}\hbox{%
$\infty42|24\infty22\infty2|2\infty22\rtimes D_{2}$%
}%
}%
\hfill\copy\matricesbox
}%
\ifdim\wd\onelinebox>\myboxwidth
\hbox to \myboxwidth{%
$\Pi_{14,21}$ spans $L_{142.20}$%
\hfil
$\infty42|24\infty22\infty2|2\infty22\rtimes D_{2}$%
}%
\box\matricesbox
\else
\hbox to \myboxwidth{%
\unhbox\onelinebox
}%
\fi
\else
\hbox to \myboxwidth{%
$\Pi_{14,21}$ spans $L_{142.20}$%
\hfil}%
\hbox to \myboxwidth{%
$\infty42|24\infty22\infty2|2\infty22\rtimes D_{2}$%
\hfil}%
\box\matricesbox
\fi
}%
\hfill\discretionary{}{}{}%
\setbox\matricesbox=\hbox{%
{$\left[\!\llap{\phantom{%
\begingroup \smaller\smaller\smaller\begin{tabular}{@{}c@{}}%
\phantom{0}\\\phantom{0}\\\phantom{0}
\end{tabular}\endgroup%
}}\right.$}%
\begingroup \smaller\smaller\smaller\begin{tabular}{@{}c@{}}%
-1\\\phantom{0}\\\phantom{0}
\end{tabular}\endgroup%
\kern3pt%
\begingroup \smaller\smaller\smaller\begin{tabular}{@{}c@{}}%
\phantom{0}\\9\\\phantom{0}
\end{tabular}\endgroup%
\kern3pt%
\begingroup \smaller\smaller\smaller\begin{tabular}{@{}c@{}}%
\phantom{0}\\\phantom{0}\\9
\end{tabular}\endgroup%
{$\left.\llap{\phantom{%
\begingroup \smaller\smaller\smaller\begin{tabular}{@{}c@{}}%
\phantom{0}\\\phantom{0}\\\phantom{0}
\end{tabular}\endgroup%
}}\!\right]$}%
{$\left[\!\llap{\phantom{%
\begingroup \smaller\smaller\smaller\begin{tabular}{@{}c@{}}%
0\\0\\0
\end{tabular}\endgroup%
}}\right.$}%
\begingroup \smaller\smaller\smaller\begin{tabular}{@{}c@{}}%
9\\-3\\-1
\end{tabular}\endgroup%
\kern3pt%
\begingroup \smaller\smaller\smaller\begin{tabular}{@{}c@{}}%
8\\-2\\-2
\end{tabular}\endgroup%
\kern3pt%
\begingroup \smaller\smaller\smaller\begin{tabular}{@{}c@{}}%
9\\-1\\-3
\end{tabular}\endgroup%
\kern3pt%
\begingroup \smaller\smaller\smaller\begin{tabular}{@{}c@{}}%
9\\1\\-3
\end{tabular}\endgroup%
\kern3pt%
\begingroup \smaller\smaller\smaller\begin{tabular}{@{}c@{}}%
8\\2\\-2
\end{tabular}\endgroup%
\kern3pt%
\begingroup \smaller\smaller\smaller\begin{tabular}{@{}c@{}}%
72\\22\\-10
\end{tabular}\endgroup%
\kern3pt%
\begingroup \smaller\smaller\smaller\begin{tabular}{@{}c@{}}%
36\\12\\-2
\end{tabular}\endgroup%
{$\left.\llap{\phantom{%
\begingroup \smaller\smaller\smaller\begin{tabular}{@{}c@{}}%
0\\0\\0
\end{tabular}\endgroup%
}}\!\right]$}%
}%
\ifdim\wd\matricesbox>\halfwidth\myboxwidth=\hsize\else\myboxwidth=\halfwidth\fi
\vbox{%
\ifdim\myboxwidth=\hsize
\setbox\onelinebox=\hbox{%
\vbox{\hbox{%
$\Pi_{14,22}$ spans $L_{142.20}$%
}\hbox{%
$422\infty22\slashinfty22\infty224\slashinfty\rtimes D_{2}$%
}%
}%
\hfill\copy\matricesbox
}%
\ifdim\wd\onelinebox>\myboxwidth
\hbox to \myboxwidth{%
$\Pi_{14,22}$ spans $L_{142.20}$%
\hfil
$422\infty22\slashinfty22\infty224\slashinfty\rtimes D_{2}$%
}%
\box\matricesbox
\else
\hbox to \myboxwidth{%
\unhbox\onelinebox
}%
\fi
\else
\hbox to \myboxwidth{%
$\Pi_{14,22}$ spans $L_{142.20}$%
\hfil}%
\hbox to \myboxwidth{%
$422\infty22\slashinfty22\infty224\slashinfty\rtimes D_{2}$%
\hfil}%
\box\matricesbox
\fi
}%
\hfill\discretionary{}{}{}%
\setbox\matricesbox=\hbox{%
{$\left[\!\llap{\phantom{%
\begingroup \smaller\smaller\smaller\begin{tabular}{@{}c@{}}%
\phantom{0}\\\phantom{0}\\\phantom{0}
\end{tabular}\endgroup%
}}\right.$}%
\begingroup \smaller\smaller\smaller\begin{tabular}{@{}c@{}}%
-1\\\phantom{0}\\\phantom{0}
\end{tabular}\endgroup%
\kern3pt%
\begingroup \smaller\smaller\smaller\begin{tabular}{@{}c@{}}%
\phantom{0}\\18\\\phantom{0}
\end{tabular}\endgroup%
\kern3pt%
\begingroup \smaller\smaller\smaller\begin{tabular}{@{}c@{}}%
\phantom{0}\\\phantom{0}\\18
\end{tabular}\endgroup%
{$\left.\llap{\phantom{%
\begingroup \smaller\smaller\smaller\begin{tabular}{@{}c@{}}%
\phantom{0}\\\phantom{0}\\\phantom{0}
\end{tabular}\endgroup%
}}\!\right]$}%
{$\left[\!\llap{\phantom{%
\begingroup \smaller\smaller\smaller\begin{tabular}{@{}c@{}}%
0\\0\\0
\end{tabular}\endgroup%
}}\right.$}%
\begingroup \smaller\smaller\smaller\begin{tabular}{@{}c@{}}%
9\\1\\2
\end{tabular}\endgroup%
\kern3pt%
\begingroup \smaller\smaller\smaller\begin{tabular}{@{}c@{}}%
36\\7\\5
\end{tabular}\endgroup%
\kern3pt%
\begingroup \smaller\smaller\smaller\begin{tabular}{@{}c@{}}%
72\\16\\6
\end{tabular}\endgroup%
\kern3pt%
\begingroup \smaller\smaller\smaller\begin{tabular}{@{}c@{}}%
8\\2\\0
\end{tabular}\endgroup%
\kern3pt%
\begingroup \smaller\smaller\smaller\begin{tabular}{@{}c@{}}%
9\\2\\-1
\end{tabular}\endgroup%
\kern3pt%
\begingroup \smaller\smaller\smaller\begin{tabular}{@{}c@{}}%
9\\1\\-2
\end{tabular}\endgroup%
\kern3pt%
\begingroup \smaller\smaller\smaller\begin{tabular}{@{}c@{}}%
8\\0\\-2
\end{tabular}\endgroup%
{$\left.\llap{\phantom{%
\begingroup \smaller\smaller\smaller\begin{tabular}{@{}c@{}}%
0\\0\\0
\end{tabular}\endgroup%
}}\!\right]$}%
}%
\ifdim\wd\matricesbox>\halfwidth\myboxwidth=\hsize\else\myboxwidth=\halfwidth\fi
\vbox{%
\ifdim\myboxwidth=\hsize
\setbox\onelinebox=\hbox{%
\vbox{\hbox{%
$\Pi_{14,23}$ spans $L_{142.20}$%
}\hbox{%
$\infty422\infty22\infty422\infty22\rtimes C_{2}$%
}%
}%
\hfill\copy\matricesbox
}%
\ifdim\wd\onelinebox>\myboxwidth
\hbox to \myboxwidth{%
$\Pi_{14,23}$ spans $L_{142.20}$%
\hfil
$\infty422\infty22\infty422\infty22\rtimes C_{2}$%
}%
\box\matricesbox
\else
\hbox to \myboxwidth{%
\unhbox\onelinebox
}%
\fi
\else
\hbox to \myboxwidth{%
$\Pi_{14,23}$ spans $L_{142.20}$%
\hfil}%
\hbox to \myboxwidth{%
$\infty422\infty22\infty422\infty22\rtimes C_{2}$%
\hfil}%
\box\matricesbox
\fi
}%
\hfill\discretionary{}{}{}%
\setbox\matricesbox=\hbox{%
{$\left[\!\llap{\phantom{%
\begingroup \smaller\smaller\smaller
\begin{tabular}{@{}c@{}}%
\phantom{0}\\\phantom{0}\\\phantom{0}
\end{tabular}\endgroup%
}}\right.$}%
\begingroup \smaller\smaller\smaller\begin{tabular}{@{}c@{}}%
-1\\\phantom{0}\\\phantom{0}
\end{tabular}\endgroup%
\kern3pt%
\begingroup \smaller\smaller\smaller\begin{tabular}{@{}c@{}}%
\phantom{0}\\18\\\phantom{0}
\end{tabular}\endgroup%
\kern3pt%
\begingroup \smaller\smaller\smaller\begin{tabular}{@{}c@{}}%
\phantom{0}\\\phantom{0}\\18
\end{tabular}\endgroup%
{$\left.\llap{\phantom{%
\begingroup \smaller\smaller\smaller\begin{tabular}{@{}c@{}}%
\phantom{0}\\\phantom{0}\\\phantom{0}
\end{tabular}\endgroup%
}}\!\right]$}%
{$\left[\!\llap{\phantom{%
\begingroup \smaller\smaller\smaller\begin{tabular}{@{}c@{}}%
0\\0\\0
\end{tabular}\endgroup%
}}\right.$}%
\begingroup \smaller\smaller\smaller\begin{tabular}{@{}c@{}}%
9\\1\\2
\end{tabular}\endgroup%
\kern3pt%
\begingroup \smaller\smaller\smaller\begin{tabular}{@{}c@{}}%
36\\7\\5
\end{tabular}\endgroup%
\kern3pt%
\begingroup \smaller\smaller\smaller\begin{tabular}{@{}c@{}}%
72\\16\\6
\end{tabular}\endgroup%
\kern3pt%
\begingroup \smaller\smaller\smaller\begin{tabular}{@{}c@{}}%
8\\2\\0
\end{tabular}\endgroup%
\kern3pt%
\begingroup \smaller\smaller\smaller\begin{tabular}{@{}c@{}}%
9\\2\\-1
\end{tabular}\endgroup%
\kern3pt%
\begingroup \smaller\smaller\smaller\begin{tabular}{@{}c@{}}%
9\\1\\-2
\end{tabular}\endgroup%
\kern3pt%
\begingroup \smaller\smaller\smaller\begin{tabular}{@{}c@{}}%
8\\0\\-2
\end{tabular}\endgroup%
\kern3pt%
\begingroup \smaller\smaller\smaller\begin{tabular}{@{}c@{}}%
9\\-1\\-2
\end{tabular}\endgroup%
\kern3pt%
\begingroup \smaller\smaller\smaller\begin{tabular}{@{}c@{}}%
9\\-2\\-1
\end{tabular}\endgroup%
\kern3pt%
\begingroup \smaller\smaller\smaller\begin{tabular}{@{}c@{}}%
8\\-2\\0
\end{tabular}\endgroup%
\kern3pt%
\begingroup \smaller\smaller\smaller\begin{tabular}{@{}c@{}}%
9\\-2\\1
\end{tabular}\endgroup%
\kern3pt%
\begingroup \smaller\smaller\smaller\begin{tabular}{@{}c@{}}%
36\\-5\\7
\end{tabular}\endgroup%
\kern3pt%
\begingroup \smaller\smaller\smaller\begin{tabular}{@{}c@{}}%
72\\-6\\16
\end{tabular}\endgroup%
\kern3pt%
\begingroup \smaller\smaller\smaller\begin{tabular}{@{}c@{}}%
8\\0\\2
\end{tabular}\endgroup%
{$\left.\llap{\phantom{%
\begingroup \smaller\smaller\smaller\begin{tabular}{@{}c@{}}%
0\\0\\0
\end{tabular}\endgroup%
}}\!\right]$}%
}%
\ifdim\wd\matricesbox>\halfwidth\myboxwidth=\hsize\else\myboxwidth=\halfwidth\fi
\vbox{%
\ifdim\myboxwidth=\hsize
\setbox\onelinebox=\hbox{%
\vbox{\hbox{%
$\Pi_{14,24}$ spans $L_{142.20}$%
}\hbox{%
$\infty422\infty22\infty22\infty422$%
}%
}%
\hfill\copy\matricesbox
}%
\ifdim\wd\onelinebox>\myboxwidth
\hbox to \myboxwidth{%
$\Pi_{14,24}$ spans $L_{142.20}$%
\hfil
$\infty422\infty22\infty22\infty422$%
}%
\box\matricesbox
\else
\hbox to \myboxwidth{%
\unhbox\onelinebox
}%
\fi
\else
\hbox to \myboxwidth{%
$\Pi_{14,24}$ spans $L_{142.20}$%
\hfil}%
\hbox to \myboxwidth{%
$\infty422\infty22\infty22\infty422$%
\hfil}%
\box\matricesbox
\fi
}%
\hfill\discretionary{}{}{}%
\setbox\matricesbox=\hbox{%
{$\left[\!\llap{\phantom{%
\begingroup \smaller\smaller\smaller
\begin{tabular}{@{}c@{}}%
\phantom{0}\\\phantom{0}\\\phantom{0}
\end{tabular}\endgroup%
}}\right.$}%
\begingroup \smaller\smaller\smaller\begin{tabular}{@{}c@{}}%
-1\\\phantom{0}\\\phantom{0}
\end{tabular}\endgroup%
\kern3pt%
\begingroup \smaller\smaller\smaller\begin{tabular}{@{}c@{}}%
\phantom{0}\\30\\-15
\end{tabular}\endgroup%
\kern3pt%
\begingroup \smaller\smaller\smaller\begin{tabular}{@{}c@{}}%
\phantom{0}\\-15\\30
\end{tabular}\endgroup%
{$\left.\llap{\phantom{%
\begingroup \smaller\smaller\smaller\begin{tabular}{@{}c@{}}%
\phantom{0}\\\phantom{0}\\\phantom{0}
\end{tabular}\endgroup%
}}\!\right]$}%
{$\left[\!\llap{\phantom{%
\begingroup \smaller\smaller\smaller\begin{tabular}{@{}c@{}}%
0\\0\\0
\end{tabular}\endgroup%
}}\right.$}%
\begingroup \smaller\smaller\smaller\begin{tabular}{@{}c@{}}%
5\\0\\1
\end{tabular}\endgroup%
\kern3pt%
\begingroup \smaller\smaller\smaller\begin{tabular}{@{}c@{}}%
9\\-1\\1
\end{tabular}\endgroup%
\kern3pt%
\begingroup \smaller\smaller\smaller\begin{tabular}{@{}c@{}}%
5\\-1\\0
\end{tabular}\endgroup%
\kern3pt%
\begingroup \smaller\smaller\smaller\begin{tabular}{@{}c@{}}%
9\\-2\\-1
\end{tabular}\endgroup%
\kern3pt%
\begingroup \smaller\smaller\smaller\begin{tabular}{@{}c@{}}%
5\\-1\\-1
\end{tabular}\endgroup%
\kern3pt%
\begingroup \smaller\smaller\smaller\begin{tabular}{@{}c@{}}%
9\\-1\\-2
\end{tabular}\endgroup%
\kern3pt%
\begingroup \smaller\smaller\smaller\begin{tabular}{@{}c@{}}%
5\\0\\-1
\end{tabular}\endgroup%
\kern3pt%
\begingroup \smaller\smaller\smaller\begin{tabular}{@{}c@{}}%
9\\1\\-1
\end{tabular}\endgroup%
\kern3pt%
\begingroup \smaller\smaller\smaller\begin{tabular}{@{}c@{}}%
5\\1\\0
\end{tabular}\endgroup%
\kern3pt%
\begingroup \smaller\smaller\smaller\begin{tabular}{@{}c@{}}%
9\\2\\1
\end{tabular}\endgroup%
\kern3pt%
\begingroup \smaller\smaller\smaller\begin{tabular}{@{}c@{}}%
30\\6\\5
\end{tabular}\endgroup%
\kern3pt%
\begingroup \smaller\smaller\smaller\begin{tabular}{@{}c@{}}%
30\\5\\6
\end{tabular}\endgroup%
\kern3pt%
\begingroup \smaller\smaller\smaller\begin{tabular}{@{}c@{}}%
90\\11\\19
\end{tabular}\endgroup%
\kern3pt%
\begingroup \smaller\smaller\smaller\begin{tabular}{@{}c@{}}%
90\\8\\19
\end{tabular}\endgroup%
{$\left.\llap{\phantom{%
\begingroup \smaller\smaller\smaller\begin{tabular}{@{}c@{}}%
0\\0\\0
\end{tabular}\endgroup%
}}\!\right]$}%
}%
\ifdim\wd\matricesbox>\halfwidth\myboxwidth=\hsize\else\myboxwidth=\halfwidth\fi
\vbox{%
\ifdim\myboxwidth=\hsize
\setbox\onelinebox=\hbox{%
\vbox{\hbox{%
$\Pi_{14,25}$ spans $L_{16.13}$%
}\hbox{%
$22222222223632$%
}%
}%
\hfill\copy\matricesbox
}%
\ifdim\wd\onelinebox>\myboxwidth
\hbox to \myboxwidth{%
$\Pi_{14,25}$ spans $L_{16.13}$%
\hfil
$22222222223632$%
}%
\box\matricesbox
\else
\hbox to \myboxwidth{%
\unhbox\onelinebox
}%
\fi
\else
\hbox to \myboxwidth{%
$\Pi_{14,25}$ spans $L_{16.13}$%
\hfil}%
\hbox to \myboxwidth{%
$22222222223632$%
\hfil}%
\box\matricesbox
\fi
}%
\hfill\discretionary{}{}{}%
\setbox\matricesbox=\hbox{%
{$\left[\!\llap{\phantom{%
\begingroup \smaller\smaller\smaller
\begin{tabular}{@{}c@{}}%
\phantom{0}\\\phantom{0}\\\phantom{0}
\end{tabular}\endgroup%
}}\right.$}%
\begingroup \smaller\smaller\smaller\begin{tabular}{@{}c@{}}%
-1\\\phantom{0}\\\phantom{0}
\end{tabular}\endgroup%
\kern3pt%
\begingroup \smaller\smaller\smaller\begin{tabular}{@{}c@{}}%
\phantom{0}\\30\\-15
\end{tabular}\endgroup%
\kern3pt%
\begingroup \smaller\smaller\smaller\begin{tabular}{@{}c@{}}%
\phantom{0}\\-15\\30
\end{tabular}\endgroup%
{$\left.\llap{\phantom{%
\begingroup \smaller\smaller\smaller\begin{tabular}{@{}c@{}}%
\phantom{0}\\\phantom{0}\\\phantom{0}
\end{tabular}\endgroup%
}}\!\right]$}%
{$\left[\!\llap{\phantom{%
\begingroup \smaller\smaller\smaller\begin{tabular}{@{}c@{}}%
0\\0\\0
\end{tabular}\endgroup%
}}\right.$}%
\begingroup \smaller\smaller\smaller\begin{tabular}{@{}c@{}}%
5\\0\\1
\end{tabular}\endgroup%
\kern3pt%
\begingroup \smaller\smaller\smaller\begin{tabular}{@{}c@{}}%
9\\-1\\1
\end{tabular}\endgroup%
\kern3pt%
\begingroup \smaller\smaller\smaller\begin{tabular}{@{}c@{}}%
5\\-1\\0
\end{tabular}\endgroup%
\kern3pt%
\begingroup \smaller\smaller\smaller\begin{tabular}{@{}c@{}}%
9\\-2\\-1
\end{tabular}\endgroup%
\kern3pt%
\begingroup \smaller\smaller\smaller\begin{tabular}{@{}c@{}}%
5\\-1\\-1
\end{tabular}\endgroup%
\kern3pt%
\begingroup \smaller\smaller\smaller\begin{tabular}{@{}c@{}}%
9\\-1\\-2
\end{tabular}\endgroup%
\kern3pt%
\begingroup \smaller\smaller\smaller\begin{tabular}{@{}c@{}}%
5\\0\\-1
\end{tabular}\endgroup%
\kern3pt%
\begingroup \smaller\smaller\smaller\begin{tabular}{@{}c@{}}%
9\\1\\-1
\end{tabular}\endgroup%
\kern3pt%
\begingroup \smaller\smaller\smaller\begin{tabular}{@{}c@{}}%
30\\5\\-1
\end{tabular}\endgroup%
\kern3pt%
\begingroup \smaller\smaller\smaller\begin{tabular}{@{}c@{}}%
30\\6\\1
\end{tabular}\endgroup%
\kern3pt%
\begingroup \smaller\smaller\smaller\begin{tabular}{@{}c@{}}%
9\\2\\1
\end{tabular}\endgroup%
\kern3pt%
\begingroup \smaller\smaller\smaller\begin{tabular}{@{}c@{}}%
5\\1\\1
\end{tabular}\endgroup%
\kern3pt%
\begingroup \smaller\smaller\smaller\begin{tabular}{@{}c@{}}%
90\\11\\19
\end{tabular}\endgroup%
\kern3pt%
\begingroup \smaller\smaller\smaller\begin{tabular}{@{}c@{}}%
90\\8\\19
\end{tabular}\endgroup%
{$\left.\llap{\phantom{%
\begingroup \smaller\smaller\smaller\begin{tabular}{@{}c@{}}%
0\\0\\0
\end{tabular}\endgroup%
}}\!\right]$}%
}%
\ifdim\wd\matricesbox>\halfwidth\myboxwidth=\hsize\else\myboxwidth=\halfwidth\fi
\vbox{%
\ifdim\myboxwidth=\hsize
\setbox\onelinebox=\hbox{%
\vbox{\hbox{%
$\Pi_{14,26}$ spans $L_{16.13}$%
}\hbox{%
$22222222322232$%
}%
}%
\hfill\copy\matricesbox
}%
\ifdim\wd\onelinebox>\myboxwidth
\hbox to \myboxwidth{%
$\Pi_{14,26}$ spans $L_{16.13}$%
\hfil
$22222222322232$%
}%
\box\matricesbox
\else
\hbox to \myboxwidth{%
\unhbox\onelinebox
}%
\fi
\else
\hbox to \myboxwidth{%
$\Pi_{14,26}$ spans $L_{16.13}$%
\hfil}%
\hbox to \myboxwidth{%
$22222222322232$%
\hfil}%
\box\matricesbox
\fi
}%
\hfill\discretionary{}{}{}%
\setbox\matricesbox=\hbox{%
{$\left[\!\llap{\phantom{%
\begingroup \smaller\smaller\smaller
\endgroup%
}}\!\right]$}%
}%
\ifdim\wd\matricesbox>\halfwidth\myboxwidth=\hsize\else\myboxwidth=\halfwidth\fi
\vbox{%
\ifdim\myboxwidth=\hsize
\setbox\onelinebox=\hbox{%
\vbox{\hbox{%
$\Pi_{14,27}$ spans $L_{16.13}$%
}\hbox{%
$22222232222232$%
}%
}%
\hfill\copy\matricesbox
}%
\ifdim\wd\onelinebox>\myboxwidth
\hbox to \myboxwidth{%
$\Pi_{14,27}$ spans $L_{16.13}$%
\hfil
$22222232222232$%
}%
\box\matricesbox
\else
\hbox to \myboxwidth{%
\unhbox\onelinebox
}%
\fi
\else
\hbox to \myboxwidth{%
$\Pi_{14,27}$ spans $L_{16.13}$%
\hfil}%
\hbox to \myboxwidth{%
$22222232222232$%
\hfil}%
\box\matricesbox
\fi
}%
\hfill\discretionary{}{}{}%
\setbox\matricesbox=\hbox{%
{$\left[\!\llap{\phantom{%
\begingroup \smaller\smaller\smaller
\endgroup%
}}\!\right]$}%
}%
\ifdim\wd\matricesbox>\halfwidth\myboxwidth=\hsize\else\myboxwidth=\halfwidth\fi
\vbox{%
\ifdim\myboxwidth=\hsize
\setbox\onelinebox=\hbox{%
\vbox{\hbox{%
$\Pi_{14,28}$ spans $L_{123.8}$%
}\hbox{%
$22422224242422$%
}%
}%
\hfill\copy\matricesbox
}%
\ifdim\wd\onelinebox>\myboxwidth
\hbox to \myboxwidth{%
$\Pi_{14,28}$ spans $L_{123.8}$%
\hfil
$22422224242422$%
}%
\box\matricesbox
\else
\hbox to \myboxwidth{%
\unhbox\onelinebox
}%
\fi
\else
\hbox to \myboxwidth{%
$\Pi_{14,28}$ spans $L_{123.8}$%
\hfil}%
\hbox to \myboxwidth{%
$22422224242422$%
\hfil}%
\box\matricesbox
\fi
}%
\hfill\discretionary{}{}{}%
\setbox\matricesbox=\hbox{%
{$\left[\!\llap{\phantom{%
\begingroup \smaller\smaller\smaller
\endgroup%
}}\!\right]$}%
}%
\ifdim\wd\matricesbox>\halfwidth\myboxwidth=\hsize\else\myboxwidth=\halfwidth\fi
\vbox{%
\ifdim\myboxwidth=\hsize
\setbox\onelinebox=\hbox{%
\vbox{\hbox{%
$\Pi_{14,29}$ spans $L_{123.8}$%
}\hbox{%
$22424242424222$%
}%
}%
\hfill\copy\matricesbox
}%
\ifdim\wd\onelinebox>\myboxwidth
\hbox to \myboxwidth{%
$\Pi_{14,29}$ spans $L_{123.8}$%
\hfil
$22424242424222$%
}%
\box\matricesbox
\else
\hbox to \myboxwidth{%
\unhbox\onelinebox
}%
\fi
\else
\hbox to \myboxwidth{%
$\Pi_{14,29}$ spans $L_{123.8}$%
\hfil}%
\hbox to \myboxwidth{%
$22424242424222$%
\hfil}%
\box\matricesbox
\fi
}%
\hfill\discretionary{}{}{}%

\vskip2pt\hrule\vskip2pt

\leavevmode\setbox\matricesbox=\hbox{%
{$\left[\!\llap{\phantom{%
\begingroup \smaller\smaller\smaller\begin{tabular}{@{}c@{}}%
\phantom{0}\\\phantom{0}\\\phantom{0}\\\phantom{0}
\end{tabular}\endgroup%
}}\right.$}%
\begingroup \smaller\smaller\smaller\begin{tabular}{@{}c@{}}%
-1\\\phantom{0}\\\phantom{0}\\\phantom{0}
\end{tabular}\endgroup%
\kern3pt%
\begingroup \smaller\smaller\smaller\begin{tabular}{@{}c@{}}%
\phantom{0}\\10/3\\\phantom{0}\\\phantom{0}
\end{tabular}\endgroup%
\kern3pt%
\begingroup \smaller\smaller\smaller\begin{tabular}{@{}c@{}}%
\phantom{0}\\\phantom{0}\\10/3\\\phantom{0}
\end{tabular}\endgroup%
\kern3pt%
\begingroup \smaller\smaller\smaller\begin{tabular}{@{}c@{}}%
\phantom{0}\\\phantom{0}\\\phantom{0}\\10/3
\end{tabular}\endgroup%
{$\left.\llap{\phantom{%
\begingroup \smaller\smaller\smaller\begin{tabular}{@{}c@{}}%
\phantom{0}\\\phantom{0}\\\phantom{0}\\\phantom{0}
\end{tabular}\endgroup%
}}\!\right]$}%
{$\left[\!\llap{\phantom{%
\begingroup \smaller\smaller\smaller\begin{tabular}{@{}c@{}}%
0\\0\\0\\0
\end{tabular}\endgroup%
}}\right.$}%
\begingroup \smaller\smaller\smaller\begin{tabular}{@{}c@{}}%
4\\-2\\1\\1
\end{tabular}\endgroup%
\kern3pt%
\begingroup \smaller\smaller\smaller\begin{tabular}{@{}c@{}}%
15\\-6\\0\\6
\end{tabular}\endgroup%
\kern3pt%
\begingroup \smaller\smaller\smaller\begin{tabular}{@{}c@{}}%
20\\-6\\-3\\9
\end{tabular}\endgroup%
{$\left.\llap{\phantom{%
\begingroup \smaller\smaller\smaller\begin{tabular}{@{}c@{}}%
0\\0\\0\\0
\end{tabular}\endgroup%
}}\!\right]$}%
}%
\ifdim\wd\matricesbox>\halfwidth\myboxwidth=\hsize\else\myboxwidth=\halfwidth\fi
\vbox{%
\ifdim\myboxwidth=\hsize
\setbox\onelinebox=\hbox{%
\vbox{\hbox{%
$\Pi_{15,1}$ spans $L_{31.7}$%
}\hbox{%
$|22\slashthree22|22\slashthree22|22\slashthree22\rtimes D_{6}$%
}%
}%
\hfill\copy\matricesbox
}%
\ifdim\wd\onelinebox>\myboxwidth
\hbox to \myboxwidth{%
$\Pi_{15,1}$ spans $L_{31.7}$%
\hfil
$|22\slashthree22|22\slashthree22|22\slashthree22\rtimes D_{6}$%
}%
\box\matricesbox
\else
\hbox to \myboxwidth{%
\unhbox\onelinebox
}%
\fi
\else
\hbox to \myboxwidth{%
$\Pi_{15,1}$ spans $L_{31.7}$%
\hfil}%
\hbox to \myboxwidth{%
$|22\slashthree22|22\slashthree22|22\slashthree22\rtimes D_{6}$%
\hfil}%
\box\matricesbox
\fi
}%
\hfill\discretionary{}{}{}%
\setbox\matricesbox=\hbox{%
{$\left[\!\llap{\phantom{%
\begingroup \smaller\smaller\smaller\begin{tabular}{@{}c@{}}%
\phantom{0}\\\phantom{0}\\\phantom{0}\\\phantom{0}
\end{tabular}\endgroup%
}}\right.$}%
\begingroup \smaller\smaller\smaller\begin{tabular}{@{}c@{}}%
-1\\\phantom{0}\\\phantom{0}\\\phantom{0}
\end{tabular}\endgroup%
\kern3pt%
\begingroup \smaller\smaller\smaller\begin{tabular}{@{}c@{}}%
\phantom{0}\\5\\\phantom{0}\\\phantom{0}
\end{tabular}\endgroup%
\kern3pt%
\begingroup \smaller\smaller\smaller\begin{tabular}{@{}c@{}}%
\phantom{0}\\\phantom{0}\\5\\\phantom{0}
\end{tabular}\endgroup%
\kern3pt%
\begingroup \smaller\smaller\smaller\begin{tabular}{@{}c@{}}%
\phantom{0}\\\phantom{0}\\\phantom{0}\\5
\end{tabular}\endgroup%
{$\left.\llap{\phantom{%
\begingroup \smaller\smaller\smaller\begin{tabular}{@{}c@{}}%
\phantom{0}\\\phantom{0}\\\phantom{0}\\\phantom{0}
\end{tabular}\endgroup%
}}\!\right]$}%
{$\left[\!\llap{\phantom{%
\begingroup \smaller\smaller\smaller\begin{tabular}{@{}c@{}}%
0\\0\\0\\0
\end{tabular}\endgroup%
}}\right.$}%
\begingroup \smaller\smaller\smaller\begin{tabular}{@{}c@{}}%
30\\11\\-4\\-7
\end{tabular}\endgroup%
\kern3pt%
\begingroup \smaller\smaller\smaller\begin{tabular}{@{}c@{}}%
9\\3\\0\\-3
\end{tabular}\endgroup%
\kern3pt%
\begingroup \smaller\smaller\smaller\begin{tabular}{@{}c@{}}%
5\\1\\1\\-2
\end{tabular}\endgroup%
{$\left.\llap{\phantom{%
\begingroup \smaller\smaller\smaller\begin{tabular}{@{}c@{}}%
0\\0\\0\\0
\end{tabular}\endgroup%
}}\!\right]$}%
}%
\ifdim\wd\matricesbox>\halfwidth\myboxwidth=\hsize\else\myboxwidth=\halfwidth\fi
\vbox{%
\ifdim\myboxwidth=\hsize
\setbox\onelinebox=\hbox{%
\vbox{\hbox{%
$\Pi_{15,2}$ spans $L_{16.13}$%
}\hbox{%
$\slashthree22|22\slashthree22|22\slashthree22|22\rtimes D_{6}$%
}%
}%
\hfill\copy\matricesbox
}%
\ifdim\wd\onelinebox>\myboxwidth
\hbox to \myboxwidth{%
$\Pi_{15,2}$ spans $L_{16.13}$%
\hfil
$\slashthree22|22\slashthree22|22\slashthree22|22\rtimes D_{6}$%
}%
\box\matricesbox
\else
\hbox to \myboxwidth{%
\unhbox\onelinebox
}%
\fi
\else
\hbox to \myboxwidth{%
$\Pi_{15,2}$ spans $L_{16.13}$%
\hfil}%
\hbox to \myboxwidth{%
$\slashthree22|22\slashthree22|22\slashthree22|22\rtimes D_{6}$%
\hfil}%
\box\matricesbox
\fi
}%
\hfill\discretionary{}{}{}%
\setbox\matricesbox=\hbox{%
{$\left[\!\llap{\phantom{%
\begingroup \smaller\smaller\smaller\begin{tabular}{@{}c@{}}%
\phantom{0}\\\phantom{0}\\\phantom{0}\\\phantom{0}
\end{tabular}\endgroup%
}}\right.$}%
\begingroup \smaller\smaller\smaller\begin{tabular}{@{}c@{}}%
-1\\\phantom{0}\\\phantom{0}\\\phantom{0}
\end{tabular}\endgroup%
\kern3pt%
\begingroup \smaller\smaller\smaller\begin{tabular}{@{}c@{}}%
\phantom{0}\\7\\\phantom{0}\\\phantom{0}
\end{tabular}\endgroup%
\kern3pt%
\begingroup \smaller\smaller\smaller\begin{tabular}{@{}c@{}}%
\phantom{0}\\\phantom{0}\\7\\\phantom{0}
\end{tabular}\endgroup%
\kern3pt%
\begingroup \smaller\smaller\smaller\begin{tabular}{@{}c@{}}%
\phantom{0}\\\phantom{0}\\\phantom{0}\\7
\end{tabular}\endgroup%
{$\left.\llap{\phantom{%
\begingroup \smaller\smaller\smaller\begin{tabular}{@{}c@{}}%
\phantom{0}\\\phantom{0}\\\phantom{0}\\\phantom{0}
\end{tabular}\endgroup%
}}\!\right]$}%
{$\left[\!\llap{\phantom{%
\begingroup \smaller\smaller\smaller\begin{tabular}{@{}c@{}}%
0\\0\\0\\0
\end{tabular}\endgroup%
}}\right.$}%
\begingroup \smaller\smaller\smaller\begin{tabular}{@{}c@{}}%
42\\13\\-5\\-8
\end{tabular}\endgroup%
\kern3pt%
\begingroup \smaller\smaller\smaller\begin{tabular}{@{}c@{}}%
7\\2\\0\\-2
\end{tabular}\endgroup%
\kern3pt%
\begingroup \smaller\smaller\smaller\begin{tabular}{@{}c@{}}%
6\\1\\1\\-2
\end{tabular}\endgroup%
{$\left.\llap{\phantom{%
\begingroup \smaller\smaller\smaller\begin{tabular}{@{}c@{}}%
0\\0\\0\\0
\end{tabular}\endgroup%
}}\!\right]$}%
}%
\ifdim\wd\matricesbox>\halfwidth\myboxwidth=\hsize\else\myboxwidth=\halfwidth\fi
\vbox{%
\ifdim\myboxwidth=\hsize
\setbox\onelinebox=\hbox{%
\vbox{\hbox{%
$\Pi_{15,3}$ spans $L_{22.4}$%
}\hbox{%
$\slashthree22|22\slashthree22|22\slashthree22|22\rtimes D_{6}$%
}%
}%
\hfill\copy\matricesbox
}%
\ifdim\wd\onelinebox>\myboxwidth
\hbox to \myboxwidth{%
$\Pi_{15,3}$ spans $L_{22.4}$%
\hfil
$\slashthree22|22\slashthree22|22\slashthree22|22\rtimes D_{6}$%
}%
\box\matricesbox
\else
\hbox to \myboxwidth{%
\unhbox\onelinebox
}%
\fi
\else
\hbox to \myboxwidth{%
$\Pi_{15,3}$ spans $L_{22.4}$%
\hfil}%
\hbox to \myboxwidth{%
$\slashthree22|22\slashthree22|22\slashthree22|22\rtimes D_{6}$%
\hfil}%
\box\matricesbox
\fi
}%
\hfill\discretionary{}{}{}%
\setbox\matricesbox=\hbox{%
{$\left[\!\llap{\phantom{%
\begingroup \smaller\smaller\smaller\begin{tabular}{@{}c@{}}%
\phantom{0}\\\phantom{0}\\\phantom{0}\\\phantom{0}
\end{tabular}\endgroup%
}}\right.$}%
\begingroup \smaller\smaller\smaller\begin{tabular}{@{}c@{}}%
-1\\\phantom{0}\\\phantom{0}\\\phantom{0}
\end{tabular}\endgroup%
\kern3pt%
\begingroup \smaller\smaller\smaller\begin{tabular}{@{}c@{}}%
\phantom{0}\\15\\\phantom{0}\\\phantom{0}
\end{tabular}\endgroup%
\kern3pt%
\begingroup \smaller\smaller\smaller\begin{tabular}{@{}c@{}}%
\phantom{0}\\\phantom{0}\\15\\\phantom{0}
\end{tabular}\endgroup%
\kern3pt%
\begingroup \smaller\smaller\smaller\begin{tabular}{@{}c@{}}%
\phantom{0}\\\phantom{0}\\\phantom{0}\\15
\end{tabular}\endgroup%
{$\left.\llap{\phantom{%
\begingroup \smaller\smaller\smaller\begin{tabular}{@{}c@{}}%
\phantom{0}\\\phantom{0}\\\phantom{0}\\\phantom{0}
\end{tabular}\endgroup%
}}\!\right]$}%
{$\left[\!\llap{\phantom{%
\begingroup \smaller\smaller\smaller\begin{tabular}{@{}c@{}}%
0\\0\\0\\0
\end{tabular}\endgroup%
}}\right.$}%
\begingroup \smaller\smaller\smaller\begin{tabular}{@{}c@{}}%
90\\-19\\8\\11
\end{tabular}\endgroup%
\kern3pt%
\begingroup \smaller\smaller\smaller\begin{tabular}{@{}c@{}}%
5\\-1\\0\\1
\end{tabular}\endgroup%
\kern3pt%
\begingroup \smaller\smaller\smaller\begin{tabular}{@{}c@{}}%
9\\-1\\-1\\2
\end{tabular}\endgroup%
{$\left.\llap{\phantom{%
\begingroup \smaller\smaller\smaller\begin{tabular}{@{}c@{}}%
0\\0\\0\\0
\end{tabular}\endgroup%
}}\!\right]$}%
}%
\ifdim\wd\matricesbox>\halfwidth\myboxwidth=\hsize\else\myboxwidth=\halfwidth\fi
\vbox{%
\ifdim\myboxwidth=\hsize
\setbox\onelinebox=\hbox{%
\vbox{\hbox{%
$\Pi_{15,4}$ spans $L_{16.13}$%
}\hbox{%
$\slashthree22|22\slashthree22|22\slashthree22|22\rtimes D_{6}$%
}%
}%
\hfill\copy\matricesbox
}%
\ifdim\wd\onelinebox>\myboxwidth
\hbox to \myboxwidth{%
$\Pi_{15,4}$ spans $L_{16.13}$%
\hfil
$\slashthree22|22\slashthree22|22\slashthree22|22\rtimes D_{6}$%
}%
\box\matricesbox
\else
\hbox to \myboxwidth{%
\unhbox\onelinebox
}%
\fi
\else
\hbox to \myboxwidth{%
$\Pi_{15,4}$ spans $L_{16.13}$%
\hfil}%
\hbox to \myboxwidth{%
$\slashthree22|22\slashthree22|22\slashthree22|22\rtimes D_{6}$%
\hfil}%
\box\matricesbox
\fi
}%
\hfill\discretionary{}{}{}%
\setbox\matricesbox=\hbox{%
{$\left[\!\llap{\phantom{%
\begingroup \smaller\smaller\smaller\begin{tabular}{@{}c@{}}%
\phantom{0}\\\phantom{0}\\\phantom{0}
\end{tabular}\endgroup%
}}\right.$}%
\begingroup \smaller\smaller\smaller\begin{tabular}{@{}c@{}}%
-1\\\phantom{0}\\\phantom{0}
\end{tabular}\endgroup%
\kern3pt%
\begingroup \smaller\smaller\smaller\begin{tabular}{@{}c@{}}%
\phantom{0}\\5\\\phantom{0}
\end{tabular}\endgroup%
\kern3pt%
\begingroup \smaller\smaller\smaller\begin{tabular}{@{}c@{}}%
\phantom{0}\\\phantom{0}\\15
\end{tabular}\endgroup%
{$\left.\llap{\phantom{%
\begingroup \smaller\smaller\smaller\begin{tabular}{@{}c@{}}%
\phantom{0}\\\phantom{0}\\\phantom{0}
\end{tabular}\endgroup%
}}\!\right]$}%
{$\left[\!\llap{\phantom{%
\begingroup \smaller\smaller\smaller\begin{tabular}{@{}c@{}}%
0\\0\\0
\end{tabular}\endgroup%
}}\right.$}%
\begingroup \smaller\smaller\smaller\begin{tabular}{@{}c@{}}%
4\\2\\0
\end{tabular}\endgroup%
\kern3pt%
\begingroup \smaller\smaller\smaller\begin{tabular}{@{}c@{}}%
15\\6\\2
\end{tabular}\endgroup%
\kern3pt%
\begingroup \smaller\smaller\smaller\begin{tabular}{@{}c@{}}%
4\\1\\1
\end{tabular}\endgroup%
\kern3pt%
\begingroup \smaller\smaller\smaller\begin{tabular}{@{}c@{}}%
15\\0\\4
\end{tabular}\endgroup%
\kern3pt%
\begingroup \smaller\smaller\smaller\begin{tabular}{@{}c@{}}%
20\\-3\\5
\end{tabular}\endgroup%
\kern3pt%
\begingroup \smaller\smaller\smaller\begin{tabular}{@{}c@{}}%
20\\-6\\4
\end{tabular}\endgroup%
\kern3pt%
\begingroup \smaller\smaller\smaller\begin{tabular}{@{}c@{}}%
15\\-6\\2
\end{tabular}\endgroup%
\kern3pt%
\begingroup \smaller\smaller\smaller\begin{tabular}{@{}c@{}}%
20\\-9\\1
\end{tabular}\endgroup%
{$\left.\llap{\phantom{%
\begingroup \smaller\smaller\smaller\begin{tabular}{@{}c@{}}%
0\\0\\0
\end{tabular}\endgroup%
}}\!\right]$}%
}%
\ifdim\wd\matricesbox>\halfwidth\myboxwidth=\hsize\else\myboxwidth=\halfwidth\fi
\vbox{%
\ifdim\myboxwidth=\hsize
\setbox\onelinebox=\hbox{%
\vbox{\hbox{%
$\Pi_{15,5}$ spans $L_{31.7}$%
}\hbox{%
$22|2222322\slashthree22322\rtimes D_{2}$%
}%
}%
\hfill\copy\matricesbox
}%
\ifdim\wd\onelinebox>\myboxwidth
\hbox to \myboxwidth{%
$\Pi_{15,5}$ spans $L_{31.7}$%
\hfil
$22|2222322\slashthree22322\rtimes D_{2}$%
}%
\box\matricesbox
\else
\hbox to \myboxwidth{%
\unhbox\onelinebox
}%
\fi
\else
\hbox to \myboxwidth{%
$\Pi_{15,5}$ spans $L_{31.7}$%
\hfil}%
\hbox to \myboxwidth{%
$22|2222322\slashthree22322\rtimes D_{2}$%
\hfil}%
\box\matricesbox
\fi
}%
\hfill\discretionary{}{}{}%
\setbox\matricesbox=\hbox{%
{$\left[\!\llap{\phantom{%
\begingroup \smaller\smaller\smaller\begin{tabular}{@{}c@{}}%
\phantom{0}\\\phantom{0}\\\phantom{0}
\end{tabular}\endgroup%
}}\right.$}%
\begingroup \smaller\smaller\smaller\begin{tabular}{@{}c@{}}%
-1\\\phantom{0}\\\phantom{0}
\end{tabular}\endgroup%
\kern3pt%
\begingroup \smaller\smaller\smaller\begin{tabular}{@{}c@{}}%
\phantom{0}\\15/2\\\phantom{0}
\end{tabular}\endgroup%
\kern3pt%
\begingroup \smaller\smaller\smaller\begin{tabular}{@{}c@{}}%
\phantom{0}\\\phantom{0}\\45/2
\end{tabular}\endgroup%
{$\left.\llap{\phantom{%
\begingroup \smaller\smaller\smaller\begin{tabular}{@{}c@{}}%
\phantom{0}\\\phantom{0}\\\phantom{0}
\end{tabular}\endgroup%
}}\!\right]$}%
{$\left[\!\llap{\phantom{%
\begingroup \smaller\smaller\smaller\begin{tabular}{@{}c@{}}%
0\\0\\0
\end{tabular}\endgroup%
}}\right.$}%
\begingroup \smaller\smaller\smaller\begin{tabular}{@{}c@{}}%
5\\-2\\0
\end{tabular}\endgroup%
\kern3pt%
\begingroup \smaller\smaller\smaller\begin{tabular}{@{}c@{}}%
9\\-3\\1
\end{tabular}\endgroup%
\kern3pt%
\begingroup \smaller\smaller\smaller\begin{tabular}{@{}c@{}}%
5\\-1\\1
\end{tabular}\endgroup%
\kern3pt%
\begingroup \smaller\smaller\smaller\begin{tabular}{@{}c@{}}%
9\\0\\2
\end{tabular}\endgroup%
\kern3pt%
\begingroup \smaller\smaller\smaller\begin{tabular}{@{}c@{}}%
5\\1\\1
\end{tabular}\endgroup%
\kern3pt%
\begingroup \smaller\smaller\smaller\begin{tabular}{@{}c@{}}%
90\\27\\11
\end{tabular}\endgroup%
\kern3pt%
\begingroup \smaller\smaller\smaller\begin{tabular}{@{}c@{}}%
90\\30\\8
\end{tabular}\endgroup%
\kern3pt%
\begingroup \smaller\smaller\smaller\begin{tabular}{@{}c@{}}%
30\\11\\1
\end{tabular}\endgroup%
{$\left.\llap{\phantom{%
\begingroup \smaller\smaller\smaller\begin{tabular}{@{}c@{}}%
0\\0\\0
\end{tabular}\endgroup%
}}\!\right]$}%
}%
\ifdim\wd\matricesbox>\halfwidth\myboxwidth=\hsize\else\myboxwidth=\halfwidth\fi
\vbox{%
\ifdim\myboxwidth=\hsize
\setbox\onelinebox=\hbox{%
\vbox{\hbox{%
$\Pi_{15,6}$ spans $L_{16.13}$%
}\hbox{%
$2222|2222236\slashthree632\rtimes D_{2}$%
}%
}%
\hfill\copy\matricesbox
}%
\ifdim\wd\onelinebox>\myboxwidth
\hbox to \myboxwidth{%
$\Pi_{15,6}$ spans $L_{16.13}$%
\hfil
$2222|2222236\slashthree632\rtimes D_{2}$%
}%
\box\matricesbox
\else
\hbox to \myboxwidth{%
\unhbox\onelinebox
}%
\fi
\else
\hbox to \myboxwidth{%
$\Pi_{15,6}$ spans $L_{16.13}$%
\hfil}%
\hbox to \myboxwidth{%
$2222|2222236\slashthree632\rtimes D_{2}$%
\hfil}%
\box\matricesbox
\fi
}%
\hfill\discretionary{}{}{}%
\setbox\matricesbox=\hbox{%
{$\left[\!\llap{\phantom{%
\begingroup \smaller\smaller\smaller\begin{tabular}{@{}c@{}}%
\phantom{0}\\\phantom{0}\\\phantom{0}
\end{tabular}\endgroup%
}}\right.$}%
\begingroup \smaller\smaller\smaller\begin{tabular}{@{}c@{}}%
-1\\\phantom{0}\\\phantom{0}
\end{tabular}\endgroup%
\kern3pt%
\begingroup \smaller\smaller\smaller\begin{tabular}{@{}c@{}}%
\phantom{0}\\45/2\\\phantom{0}
\end{tabular}\endgroup%
\kern3pt%
\begingroup \smaller\smaller\smaller\begin{tabular}{@{}c@{}}%
\phantom{0}\\\phantom{0}\\15/2
\end{tabular}\endgroup%
{$\left.\llap{\phantom{%
\begingroup \smaller\smaller\smaller\begin{tabular}{@{}c@{}}%
\phantom{0}\\\phantom{0}\\\phantom{0}
\end{tabular}\endgroup%
}}\!\right]$}%
{$\left[\!\llap{\phantom{%
\begingroup \smaller\smaller\smaller\begin{tabular}{@{}c@{}}%
0\\0\\0
\end{tabular}\endgroup%
}}\right.$}%
\begingroup \smaller\smaller\smaller\begin{tabular}{@{}c@{}}%
9\\2\\0
\end{tabular}\endgroup%
\kern3pt%
\begingroup \smaller\smaller\smaller\begin{tabular}{@{}c@{}}%
5\\1\\1
\end{tabular}\endgroup%
\kern3pt%
\begingroup \smaller\smaller\smaller\begin{tabular}{@{}c@{}}%
9\\1\\3
\end{tabular}\endgroup%
\kern3pt%
\begingroup \smaller\smaller\smaller\begin{tabular}{@{}c@{}}%
5\\0\\2
\end{tabular}\endgroup%
\kern3pt%
\begingroup \smaller\smaller\smaller\begin{tabular}{@{}c@{}}%
9\\-1\\3
\end{tabular}\endgroup%
\kern3pt%
\begingroup \smaller\smaller\smaller\begin{tabular}{@{}c@{}}%
30\\-5\\7
\end{tabular}\endgroup%
\kern3pt%
\begingroup \smaller\smaller\smaller\begin{tabular}{@{}c@{}}%
30\\-6\\4
\end{tabular}\endgroup%
\kern3pt%
\begingroup \smaller\smaller\smaller\begin{tabular}{@{}c@{}}%
90\\-19\\3
\end{tabular}\endgroup%
{$\left.\llap{\phantom{%
\begingroup \smaller\smaller\smaller\begin{tabular}{@{}c@{}}%
0\\0\\0
\end{tabular}\endgroup%
}}\!\right]$}%
}%
\ifdim\wd\matricesbox>\halfwidth\myboxwidth=\hsize\else\myboxwidth=\halfwidth\fi
\vbox{%
\ifdim\myboxwidth=\hsize
\setbox\onelinebox=\hbox{%
\vbox{\hbox{%
$\Pi_{15,7}$ spans $L_{16.13}$%
}\hbox{%
$222|2222236\slashthree6322\rtimes D_{2}$%
}%
}%
\hfill\copy\matricesbox
}%
\ifdim\wd\onelinebox>\myboxwidth
\hbox to \myboxwidth{%
$\Pi_{15,7}$ spans $L_{16.13}$%
\hfil
$222|2222236\slashthree6322\rtimes D_{2}$%
}%
\box\matricesbox
\else
\hbox to \myboxwidth{%
\unhbox\onelinebox
}%
\fi
\else
\hbox to \myboxwidth{%
$\Pi_{15,7}$ spans $L_{16.13}$%
\hfil}%
\hbox to \myboxwidth{%
$222|2222236\slashthree6322\rtimes D_{2}$%
\hfil}%
\box\matricesbox
\fi
}%
\hfill\discretionary{}{}{}%
\setbox\matricesbox=\hbox{%
{$\left[\!\llap{\phantom{%
\begingroup \smaller\smaller\smaller\begin{tabular}{@{}c@{}}%
\phantom{0}\\\phantom{0}\\\phantom{0}
\end{tabular}\endgroup%
}}\right.$}%
\begingroup \smaller\smaller\smaller\begin{tabular}{@{}c@{}}%
-1\\\phantom{0}\\\phantom{0}
\end{tabular}\endgroup%
\kern3pt%
\begingroup \smaller\smaller\smaller\begin{tabular}{@{}c@{}}%
\phantom{0}\\45/2\\\phantom{0}
\end{tabular}\endgroup%
\kern3pt%
\begingroup \smaller\smaller\smaller\begin{tabular}{@{}c@{}}%
\phantom{0}\\\phantom{0}\\15/2
\end{tabular}\endgroup%
{$\left.\llap{\phantom{%
\begingroup \smaller\smaller\smaller\begin{tabular}{@{}c@{}}%
\phantom{0}\\\phantom{0}\\\phantom{0}
\end{tabular}\endgroup%
}}\!\right]$}%
{$\left[\!\llap{\phantom{%
\begingroup \smaller\smaller\smaller\begin{tabular}{@{}c@{}}%
0\\0\\0
\end{tabular}\endgroup%
}}\right.$}%
\begingroup \smaller\smaller\smaller\begin{tabular}{@{}c@{}}%
9\\2\\0
\end{tabular}\endgroup%
\kern3pt%
\begingroup \smaller\smaller\smaller\begin{tabular}{@{}c@{}}%
5\\1\\1
\end{tabular}\endgroup%
\kern3pt%
\begingroup \smaller\smaller\smaller\begin{tabular}{@{}c@{}}%
9\\1\\3
\end{tabular}\endgroup%
\kern3pt%
\begingroup \smaller\smaller\smaller\begin{tabular}{@{}c@{}}%
5\\0\\2
\end{tabular}\endgroup%
\kern3pt%
\begingroup \smaller\smaller\smaller\begin{tabular}{@{}c@{}}%
90\\-8\\30
\end{tabular}\endgroup%
\kern3pt%
\begingroup \smaller\smaller\smaller\begin{tabular}{@{}c@{}}%
90\\-11\\27
\end{tabular}\endgroup%
\kern3pt%
\begingroup \smaller\smaller\smaller\begin{tabular}{@{}c@{}}%
5\\-1\\1
\end{tabular}\endgroup%
\kern3pt%
\begingroup \smaller\smaller\smaller\begin{tabular}{@{}c@{}}%
90\\-19\\3
\end{tabular}\endgroup%
{$\left.\llap{\phantom{%
\begingroup \smaller\smaller\smaller\begin{tabular}{@{}c@{}}%
0\\0\\0
\end{tabular}\endgroup%
}}\!\right]$}%
}%
\ifdim\wd\matricesbox>\halfwidth\myboxwidth=\hsize\else\myboxwidth=\halfwidth\fi
\vbox{%
\ifdim\myboxwidth=\hsize
\setbox\onelinebox=\hbox{%
\vbox{\hbox{%
$\Pi_{15,8}$ spans $L_{16.13}$%
}\hbox{%
$222|2222322\slashthree2232\rtimes D_{2}$%
}%
}%
\hfill\copy\matricesbox
}%
\ifdim\wd\onelinebox>\myboxwidth
\hbox to \myboxwidth{%
$\Pi_{15,8}$ spans $L_{16.13}$%
\hfil
$222|2222322\slashthree2232\rtimes D_{2}$%
}%
\box\matricesbox
\else
\hbox to \myboxwidth{%
\unhbox\onelinebox
}%
\fi
\else
\hbox to \myboxwidth{%
$\Pi_{15,8}$ spans $L_{16.13}$%
\hfil}%
\hbox to \myboxwidth{%
$222|2222322\slashthree2232\rtimes D_{2}$%
\hfil}%
\box\matricesbox
\fi
}%
\hfill\discretionary{}{}{}%
\setbox\matricesbox=\hbox{%
{$\left[\!\llap{\phantom{%
\begingroup \smaller\smaller\smaller\begin{tabular}{@{}c@{}}%
\phantom{0}\\\phantom{0}\\\phantom{0}
\end{tabular}\endgroup%
}}\right.$}%
\begingroup \smaller\smaller\smaller\begin{tabular}{@{}c@{}}%
-1\\\phantom{0}\\\phantom{0}
\end{tabular}\endgroup%
\kern3pt%
\begingroup \smaller\smaller\smaller\begin{tabular}{@{}c@{}}%
\phantom{0}\\15/2\\\phantom{0}
\end{tabular}\endgroup%
\kern3pt%
\begingroup \smaller\smaller\smaller\begin{tabular}{@{}c@{}}%
\phantom{0}\\\phantom{0}\\45/2
\end{tabular}\endgroup%
{$\left.\llap{\phantom{%
\begingroup \smaller\smaller\smaller\begin{tabular}{@{}c@{}}%
\phantom{0}\\\phantom{0}\\\phantom{0}
\end{tabular}\endgroup%
}}\!\right]$}%
{$\left[\!\llap{\phantom{%
\begingroup \smaller\smaller\smaller\begin{tabular}{@{}c@{}}%
0\\0\\0
\end{tabular}\endgroup%
}}\right.$}%
\begingroup \smaller\smaller\smaller\begin{tabular}{@{}c@{}}%
5\\-2\\0
\end{tabular}\endgroup%
\kern3pt%
\begingroup \smaller\smaller\smaller\begin{tabular}{@{}c@{}}%
9\\-3\\1
\end{tabular}\endgroup%
\kern3pt%
\begingroup \smaller\smaller\smaller\begin{tabular}{@{}c@{}}%
5\\-1\\1
\end{tabular}\endgroup%
\kern3pt%
\begingroup \smaller\smaller\smaller\begin{tabular}{@{}c@{}}%
9\\0\\2
\end{tabular}\endgroup%
\kern3pt%
\begingroup \smaller\smaller\smaller\begin{tabular}{@{}c@{}}%
30\\4\\6
\end{tabular}\endgroup%
\kern3pt%
\begingroup \smaller\smaller\smaller\begin{tabular}{@{}c@{}}%
30\\7\\5
\end{tabular}\endgroup%
\kern3pt%
\begingroup \smaller\smaller\smaller\begin{tabular}{@{}c@{}}%
9\\3\\1
\end{tabular}\endgroup%
\kern3pt%
\begingroup \smaller\smaller\smaller\begin{tabular}{@{}c@{}}%
30\\11\\1
\end{tabular}\endgroup%
{$\left.\llap{\phantom{%
\begingroup \smaller\smaller\smaller\begin{tabular}{@{}c@{}}%
0\\0\\0
\end{tabular}\endgroup%
}}\!\right]$}%
}%
\ifdim\wd\matricesbox>\halfwidth\myboxwidth=\hsize\else\myboxwidth=\halfwidth\fi
\vbox{%
\ifdim\myboxwidth=\hsize
\setbox\onelinebox=\hbox{%
\vbox{\hbox{%
$\Pi_{15,9}$ spans $L_{16.13}$%
}\hbox{%
$22|2222322\slashthree22322\rtimes D_{2}$%
}%
}%
\hfill\copy\matricesbox
}%
\ifdim\wd\onelinebox>\myboxwidth
\hbox to \myboxwidth{%
$\Pi_{15,9}$ spans $L_{16.13}$%
\hfil
$22|2222322\slashthree22322\rtimes D_{2}$%
}%
\box\matricesbox
\else
\hbox to \myboxwidth{%
\unhbox\onelinebox
}%
\fi
\else
\hbox to \myboxwidth{%
$\Pi_{15,9}$ spans $L_{16.13}$%
\hfil}%
\hbox to \myboxwidth{%
$22|2222322\slashthree22322\rtimes D_{2}$%
\hfil}%
\box\matricesbox
\fi
}%
\hfill\discretionary{}{}{}%
\setbox\matricesbox=\hbox{%
{$\left[\!\llap{\phantom{%
\begingroup \smaller\smaller\smaller\begin{tabular}{@{}c@{}}%
\phantom{0}\\\phantom{0}\\\phantom{0}
\end{tabular}\endgroup%
}}\right.$}%
\begingroup \smaller\smaller\smaller\begin{tabular}{@{}c@{}}%
-1\\\phantom{0}\\\phantom{0}
\end{tabular}\endgroup%
\kern3pt%
\begingroup \smaller\smaller\smaller\begin{tabular}{@{}c@{}}%
\phantom{0}\\15/2\\\phantom{0}
\end{tabular}\endgroup%
\kern3pt%
\begingroup \smaller\smaller\smaller\begin{tabular}{@{}c@{}}%
\phantom{0}\\\phantom{0}\\45/2
\end{tabular}\endgroup%
{$\left.\llap{\phantom{%
\begingroup \smaller\smaller\smaller\begin{tabular}{@{}c@{}}%
\phantom{0}\\\phantom{0}\\\phantom{0}
\end{tabular}\endgroup%
}}\!\right]$}%
{$\left[\!\llap{\phantom{%
\begingroup \smaller\smaller\smaller\begin{tabular}{@{}c@{}}%
0\\0\\0
\end{tabular}\endgroup%
}}\right.$}%
\begingroup \smaller\smaller\smaller\begin{tabular}{@{}c@{}}%
5\\-2\\0
\end{tabular}\endgroup%
\kern3pt%
\begingroup \smaller\smaller\smaller\begin{tabular}{@{}c@{}}%
9\\-3\\1
\end{tabular}\endgroup%
\kern3pt%
\begingroup \smaller\smaller\smaller\begin{tabular}{@{}c@{}}%
5\\-1\\1
\end{tabular}\endgroup%
\kern3pt%
\begingroup \smaller\smaller\smaller\begin{tabular}{@{}c@{}}%
90\\-3\\19
\end{tabular}\endgroup%
\kern3pt%
\begingroup \smaller\smaller\smaller\begin{tabular}{@{}c@{}}%
90\\3\\19
\end{tabular}\endgroup%
\kern3pt%
\begingroup \smaller\smaller\smaller\begin{tabular}{@{}c@{}}%
5\\1\\1
\end{tabular}\endgroup%
\kern3pt%
\begingroup \smaller\smaller\smaller\begin{tabular}{@{}c@{}}%
9\\3\\1
\end{tabular}\endgroup%
\kern3pt%
\begingroup \smaller\smaller\smaller\begin{tabular}{@{}c@{}}%
30\\11\\1
\end{tabular}\endgroup%
{$\left.\llap{\phantom{%
\begingroup \smaller\smaller\smaller\begin{tabular}{@{}c@{}}%
0\\0\\0
\end{tabular}\endgroup%
}}\!\right]$}%
}%
\ifdim\wd\matricesbox>\halfwidth\myboxwidth=\hsize\else\myboxwidth=\halfwidth\fi
\vbox{%
\ifdim\myboxwidth=\hsize
\setbox\onelinebox=\hbox{%
\vbox{\hbox{%
$\Pi_{15,10}$ spans $L_{16.13}$%
}\hbox{%
$22|2223222\slashthree22232\rtimes D_{2}$%
}%
}%
\hfill\copy\matricesbox
}%
\ifdim\wd\onelinebox>\myboxwidth
\hbox to \myboxwidth{%
$\Pi_{15,10}$ spans $L_{16.13}$%
\hfil
$22|2223222\slashthree22232\rtimes D_{2}$%
}%
\box\matricesbox
\else
\hbox to \myboxwidth{%
\unhbox\onelinebox
}%
\fi
\else
\hbox to \myboxwidth{%
$\Pi_{15,10}$ spans $L_{16.13}$%
\hfil}%
\hbox to \myboxwidth{%
$22|2223222\slashthree22232\rtimes D_{2}$%
\hfil}%
\box\matricesbox
\fi
}%
\hfill\discretionary{}{}{}%
\setbox\matricesbox=\hbox{%
{$\left[\!\llap{\phantom{%
\begingroup \smaller\smaller\smaller\begin{tabular}{@{}c@{}}%
\phantom{0}\\\phantom{0}\\\phantom{0}
\end{tabular}\endgroup%
}}\right.$}%
\begingroup \smaller\smaller\smaller\begin{tabular}{@{}c@{}}%
-1\\\phantom{0}\\\phantom{0}
\end{tabular}\endgroup%
\kern3pt%
\begingroup \smaller\smaller\smaller\begin{tabular}{@{}c@{}}%
\phantom{0}\\21/2\\\phantom{0}
\end{tabular}\endgroup%
\kern3pt%
\begingroup \smaller\smaller\smaller\begin{tabular}{@{}c@{}}%
\phantom{0}\\\phantom{0}\\7/2
\end{tabular}\endgroup%
{$\left.\llap{\phantom{%
\begingroup \smaller\smaller\smaller\begin{tabular}{@{}c@{}}%
\phantom{0}\\\phantom{0}\\\phantom{0}
\end{tabular}\endgroup%
}}\!\right]$}%
{$\left[\!\llap{\phantom{%
\begingroup \smaller\smaller\smaller\begin{tabular}{@{}c@{}}%
0\\0\\0
\end{tabular}\endgroup%
}}\right.$}%
\begingroup \smaller\smaller\smaller\begin{tabular}{@{}c@{}}%
6\\2\\0
\end{tabular}\endgroup%
\kern3pt%
\begingroup \smaller\smaller\smaller\begin{tabular}{@{}c@{}}%
7\\2\\2
\end{tabular}\endgroup%
\kern3pt%
\begingroup \smaller\smaller\smaller\begin{tabular}{@{}c@{}}%
6\\1\\3
\end{tabular}\endgroup%
\kern3pt%
\begingroup \smaller\smaller\smaller\begin{tabular}{@{}c@{}}%
7\\0\\4
\end{tabular}\endgroup%
\kern3pt%
\begingroup \smaller\smaller\smaller\begin{tabular}{@{}c@{}}%
42\\-5\\21
\end{tabular}\endgroup%
\kern3pt%
\begingroup \smaller\smaller\smaller\begin{tabular}{@{}c@{}}%
42\\-8\\18
\end{tabular}\endgroup%
\kern3pt%
\begingroup \smaller\smaller\smaller\begin{tabular}{@{}c@{}}%
7\\-2\\2
\end{tabular}\endgroup%
\kern3pt%
\begingroup \smaller\smaller\smaller\begin{tabular}{@{}c@{}}%
42\\-13\\3
\end{tabular}\endgroup%
{$\left.\llap{\phantom{%
\begingroup \smaller\smaller\smaller\begin{tabular}{@{}c@{}}%
0\\0\\0
\end{tabular}\endgroup%
}}\!\right]$}%
}%
\ifdim\wd\matricesbox>\halfwidth\myboxwidth=\hsize\else\myboxwidth=\halfwidth\fi
\vbox{%
\ifdim\myboxwidth=\hsize
\setbox\onelinebox=\hbox{%
\vbox{\hbox{%
$\Pi_{15,11}$ spans $L_{22.4}$%
}\hbox{%
$22|2222322\slashthree22322\rtimes D_{2}$%
}%
}%
\hfill\copy\matricesbox
}%
\ifdim\wd\onelinebox>\myboxwidth
\hbox to \myboxwidth{%
$\Pi_{15,11}$ spans $L_{22.4}$%
\hfil
$22|2222322\slashthree22322\rtimes D_{2}$%
}%
\box\matricesbox
\else
\hbox to \myboxwidth{%
\unhbox\onelinebox
}%
\fi
\else
\hbox to \myboxwidth{%
$\Pi_{15,11}$ spans $L_{22.4}$%
\hfil}%
\hbox to \myboxwidth{%
$22|2222322\slashthree22322\rtimes D_{2}$%
\hfil}%
\box\matricesbox
\fi
}%
\hfill\discretionary{}{}{}%
\setbox\matricesbox=\hbox{%
{$\left[\!\llap{\phantom{%
\begingroup \smaller\smaller\smaller\begin{tabular}{@{}c@{}}%
\phantom{0}\\\phantom{0}\\\phantom{0}
\end{tabular}\endgroup%
}}\right.$}%
\begingroup \smaller\smaller\smaller\begin{tabular}{@{}c@{}}%
-2\\\phantom{0}\\\phantom{0}
\end{tabular}\endgroup%
\kern3pt%
\begingroup \smaller\smaller\smaller\begin{tabular}{@{}c@{}}%
\phantom{0}\\3\\\phantom{0}
\end{tabular}\endgroup%
\kern3pt%
\begingroup \smaller\smaller\smaller\begin{tabular}{@{}c@{}}%
\phantom{0}\\\phantom{0}\\3
\end{tabular}\endgroup%
{$\left.\llap{\phantom{%
\begingroup \smaller\smaller\smaller\begin{tabular}{@{}c@{}}%
\phantom{0}\\\phantom{0}\\\phantom{0}
\end{tabular}\endgroup%
}}\!\right]$}%
{$\left[\!\llap{\phantom{%
\begingroup \smaller\smaller\smaller\begin{tabular}{@{}c@{}}%
0\\0\\0
\end{tabular}\endgroup%
}}\right.$}%
\begingroup \smaller\smaller\smaller\begin{tabular}{@{}c@{}}%
1\\1\\0
\end{tabular}\endgroup%
\kern3pt%
\begingroup \smaller\smaller\smaller\begin{tabular}{@{}c@{}}%
12\\8\\-6
\end{tabular}\endgroup%
\kern3pt%
\begingroup \smaller\smaller\smaller\begin{tabular}{@{}c@{}}%
12\\6\\-8
\end{tabular}\endgroup%
\kern3pt%
\begingroup \smaller\smaller\smaller\begin{tabular}{@{}c@{}}%
6\\1\\-5
\end{tabular}\endgroup%
\kern3pt%
\begingroup \smaller\smaller\smaller\begin{tabular}{@{}c@{}}%
6\\-1\\-5
\end{tabular}\endgroup%
\kern3pt%
\begingroup \smaller\smaller\smaller\begin{tabular}{@{}c@{}}%
12\\-6\\-8
\end{tabular}\endgroup%
\kern3pt%
\begingroup \smaller\smaller\smaller\begin{tabular}{@{}c@{}}%
12\\-8\\-6
\end{tabular}\endgroup%
\kern3pt%
\begingroup \smaller\smaller\smaller\begin{tabular}{@{}c@{}}%
6\\-5\\-1
\end{tabular}\endgroup%
{$\left.\llap{\phantom{%
\begingroup \smaller\smaller\smaller\begin{tabular}{@{}c@{}}%
0\\0\\0
\end{tabular}\endgroup%
}}\!\right]$}%
}%
\ifdim\wd\matricesbox>\halfwidth\myboxwidth=\hsize\else\myboxwidth=\halfwidth\fi
\vbox{%
\ifdim\myboxwidth=\hsize
\setbox\onelinebox=\hbox{%
\vbox{\hbox{%
$\Pi_{15,12}$ spans $L_{123.11}$%
}\hbox{%
$|2242424\slashtwo4242422\rtimes D_{2}$%
}%
}%
\hfill\copy\matricesbox
}%
\ifdim\wd\onelinebox>\myboxwidth
\hbox to \myboxwidth{%
$\Pi_{15,12}$ spans $L_{123.11}$%
\hfil
$|2242424\slashtwo4242422\rtimes D_{2}$%
}%
\box\matricesbox
\else
\hbox to \myboxwidth{%
\unhbox\onelinebox
}%
\fi
\else
\hbox to \myboxwidth{%
$\Pi_{15,12}$ spans $L_{123.11}$%
\hfil}%
\hbox to \myboxwidth{%
$|2242424\slashtwo4242422\rtimes D_{2}$%
\hfil}%
\box\matricesbox
\fi
}%
\hfill\discretionary{}{}{}%
\setbox\matricesbox=\hbox{%
{$\left[\!\llap{\phantom{%
\begingroup \smaller\smaller\smaller\begin{tabular}{@{}c@{}}%
\phantom{0}\\\phantom{0}\\\phantom{0}
\end{tabular}\endgroup%
}}\right.$}%
\begingroup \smaller\smaller\smaller\begin{tabular}{@{}c@{}}%
-3\\\phantom{0}\\\phantom{0}
\end{tabular}\endgroup%
\kern3pt%
\begingroup \smaller\smaller\smaller\begin{tabular}{@{}c@{}}%
\phantom{0}\\4\\\phantom{0}
\end{tabular}\endgroup%
\kern3pt%
\begingroup \smaller\smaller\smaller\begin{tabular}{@{}c@{}}%
\phantom{0}\\\phantom{0}\\4
\end{tabular}\endgroup%
{$\left.\llap{\phantom{%
\begingroup \smaller\smaller\smaller\begin{tabular}{@{}c@{}}%
\phantom{0}\\\phantom{0}\\\phantom{0}
\end{tabular}\endgroup%
}}\!\right]$}%
{$\left[\!\llap{\phantom{%
\begingroup \smaller\smaller\smaller\begin{tabular}{@{}c@{}}%
0\\0\\0
\end{tabular}\endgroup%
}}\right.$}%
\begingroup \smaller\smaller\smaller\begin{tabular}{@{}c@{}}%
1\\-1\\0
\end{tabular}\endgroup%
\kern3pt%
\begingroup \smaller\smaller\smaller\begin{tabular}{@{}c@{}}%
4\\-3\\2
\end{tabular}\endgroup%
\kern3pt%
\begingroup \smaller\smaller\smaller\begin{tabular}{@{}c@{}}%
4\\-2\\3
\end{tabular}\endgroup%
\kern3pt%
\begingroup \smaller\smaller\smaller\begin{tabular}{@{}c@{}}%
8\\-1\\7
\end{tabular}\endgroup%
\kern3pt%
\begingroup \smaller\smaller\smaller\begin{tabular}{@{}c@{}}%
8\\1\\7
\end{tabular}\endgroup%
\kern3pt%
\begingroup \smaller\smaller\smaller\begin{tabular}{@{}c@{}}%
4\\2\\3
\end{tabular}\endgroup%
\kern3pt%
\begingroup \smaller\smaller\smaller\begin{tabular}{@{}c@{}}%
4\\3\\2
\end{tabular}\endgroup%
\kern3pt%
\begingroup \smaller\smaller\smaller\begin{tabular}{@{}c@{}}%
8\\7\\1
\end{tabular}\endgroup%
{$\left.\llap{\phantom{%
\begingroup \smaller\smaller\smaller\begin{tabular}{@{}c@{}}%
0\\0\\0
\end{tabular}\endgroup%
}}\!\right]$}%
}%
\ifdim\wd\matricesbox>\halfwidth\myboxwidth=\hsize\else\myboxwidth=\halfwidth\fi
\vbox{%
\ifdim\myboxwidth=\hsize
\setbox\onelinebox=\hbox{%
\vbox{\hbox{%
$\Pi_{15,13}$ spans $L_{123.9}$%
}\hbox{%
$|2242424\slashtwo4242422\rtimes D_{2}$%
}%
}%
\hfill\copy\matricesbox
}%
\ifdim\wd\onelinebox>\myboxwidth
\hbox to \myboxwidth{%
$\Pi_{15,13}$ spans $L_{123.9}$%
\hfil
$|2242424\slashtwo4242422\rtimes D_{2}$%
}%
\box\matricesbox
\else
\hbox to \myboxwidth{%
\unhbox\onelinebox
}%
\fi
\else
\hbox to \myboxwidth{%
$\Pi_{15,13}$ spans $L_{123.9}$%
\hfil}%
\hbox to \myboxwidth{%
$|2242424\slashtwo4242422\rtimes D_{2}$%
\hfil}%
\box\matricesbox
\fi
}%
\hfill\discretionary{}{}{}%
\setbox\matricesbox=\hbox{%
{$\left[\!\llap{\phantom{%
\begingroup \smaller\smaller\smaller\begin{tabular}{@{}c@{}}%
\phantom{0}\\\phantom{0}\\\phantom{0}
\end{tabular}\endgroup%
}}\right.$}%
\begingroup \smaller\smaller\smaller\begin{tabular}{@{}c@{}}%
-1\\\phantom{0}\\\phantom{0}
\end{tabular}\endgroup%
\kern3pt%
\begingroup \smaller\smaller\smaller\begin{tabular}{@{}c@{}}%
\phantom{0}\\45/2\\\phantom{0}
\end{tabular}\endgroup%
\kern3pt%
\begingroup \smaller\smaller\smaller\begin{tabular}{@{}c@{}}%
\phantom{0}\\\phantom{0}\\15/2
\end{tabular}\endgroup%
{$\left.\llap{\phantom{%
\begingroup \smaller\smaller\smaller\begin{tabular}{@{}c@{}}%
\phantom{0}\\\phantom{0}\\\phantom{0}
\end{tabular}\endgroup%
}}\!\right]$}%
{$\left[\!\llap{\phantom{%
\begingroup \smaller\smaller\smaller\begin{tabular}{@{}c@{}}%
0\\0\\0
\end{tabular}\endgroup%
}}\right.$}%
\begingroup \smaller\smaller\smaller\begin{tabular}{@{}c@{}}%
9\\2\\0
\end{tabular}\endgroup%
\kern3pt%
\begingroup \smaller\smaller\smaller\begin{tabular}{@{}c@{}}%
5\\1\\-1
\end{tabular}\endgroup%
\kern3pt%
\begingroup \smaller\smaller\smaller\begin{tabular}{@{}c@{}}%
9\\1\\-3
\end{tabular}\endgroup%
\kern3pt%
\begingroup \smaller\smaller\smaller\begin{tabular}{@{}c@{}}%
30\\1\\-11
\end{tabular}\endgroup%
\kern3pt%
\begingroup \smaller\smaller\smaller\begin{tabular}{@{}c@{}}%
30\\-1\\-11
\end{tabular}\endgroup%
\kern3pt%
\begingroup \smaller\smaller\smaller\begin{tabular}{@{}c@{}}%
9\\-1\\-3
\end{tabular}\endgroup%
\kern3pt%
\begingroup \smaller\smaller\smaller\begin{tabular}{@{}c@{}}%
5\\-1\\-1
\end{tabular}\endgroup%
\kern3pt%
\begingroup \smaller\smaller\smaller\begin{tabular}{@{}c@{}}%
90\\-19\\-3
\end{tabular}\endgroup%
{$\left.\llap{\phantom{%
\begingroup \smaller\smaller\smaller\begin{tabular}{@{}c@{}}%
0\\0\\0
\end{tabular}\endgroup%
}}\!\right]$}%
}%
\ifdim\wd\matricesbox>\halfwidth\myboxwidth=\hsize\else\myboxwidth=\halfwidth\fi
\vbox{%
\ifdim\myboxwidth=\hsize
\setbox\onelinebox=\hbox{%
\vbox{\hbox{%
$\Pi_{15,14}$ spans $L_{16.13}$%
}\hbox{%
$3222|2223222\slashthree222\rtimes D_{2}$%
}%
}%
\hfill\copy\matricesbox
}%
\ifdim\wd\onelinebox>\myboxwidth
\hbox to \myboxwidth{%
$\Pi_{15,14}$ spans $L_{16.13}$%
\hfil
$3222|2223222\slashthree222\rtimes D_{2}$%
}%
\box\matricesbox
\else
\hbox to \myboxwidth{%
\unhbox\onelinebox
}%
\fi
\else
\hbox to \myboxwidth{%
$\Pi_{15,14}$ spans $L_{16.13}$%
\hfil}%
\hbox to \myboxwidth{%
$3222|2223222\slashthree222\rtimes D_{2}$%
\hfil}%
\box\matricesbox
\fi
}%
\hfill\discretionary{}{}{}%
\setbox\matricesbox=\hbox{%
{$\left[\!\llap{\phantom{%
\begingroup \smaller\smaller\smaller\begin{tabular}{@{}c@{}}%
\phantom{0}\\\phantom{0}\\\phantom{0}
\end{tabular}\endgroup%
}}\right.$}%
\begingroup \smaller\smaller\smaller\begin{tabular}{@{}c@{}}%
-1\\\phantom{0}\\\phantom{0}
\end{tabular}\endgroup%
\kern3pt%
\begingroup \smaller\smaller\smaller\begin{tabular}{@{}c@{}}%
\phantom{0}\\45/2\\\phantom{0}
\end{tabular}\endgroup%
\kern3pt%
\begingroup \smaller\smaller\smaller\begin{tabular}{@{}c@{}}%
\phantom{0}\\\phantom{0}\\15/2
\end{tabular}\endgroup%
{$\left.\llap{\phantom{%
\begingroup \smaller\smaller\smaller\begin{tabular}{@{}c@{}}%
\phantom{0}\\\phantom{0}\\\phantom{0}
\end{tabular}\endgroup%
}}\!\right]$}%
{$\left[\!\llap{\phantom{%
\begingroup \smaller\smaller\smaller\begin{tabular}{@{}c@{}}%
0\\0\\0
\end{tabular}\endgroup%
}}\right.$}%
\begingroup \smaller\smaller\smaller\begin{tabular}{@{}c@{}}%
90\\19\\3
\end{tabular}\endgroup%
\kern3pt%
\begingroup \smaller\smaller\smaller\begin{tabular}{@{}c@{}}%
5\\1\\1
\end{tabular}\endgroup%
\kern3pt%
\begingroup \smaller\smaller\smaller\begin{tabular}{@{}c@{}}%
9\\1\\3
\end{tabular}\endgroup%
\kern3pt%
\begingroup \smaller\smaller\smaller\begin{tabular}{@{}c@{}}%
5\\0\\2
\end{tabular}\endgroup%
\kern3pt%
\begingroup \smaller\smaller\smaller\begin{tabular}{@{}c@{}}%
9\\-1\\3
\end{tabular}\endgroup%
\kern3pt%
\begingroup \smaller\smaller\smaller\begin{tabular}{@{}c@{}}%
30\\-5\\7
\end{tabular}\endgroup%
\kern3pt%
\begingroup \smaller\smaller\smaller\begin{tabular}{@{}c@{}}%
30\\-6\\4
\end{tabular}\endgroup%
\kern3pt%
\begingroup \smaller\smaller\smaller\begin{tabular}{@{}c@{}}%
9\\-2\\0
\end{tabular}\endgroup%
{$\left.\llap{\phantom{%
\begingroup \smaller\smaller\smaller\begin{tabular}{@{}c@{}}%
0\\0\\0
\end{tabular}\endgroup%
}}\!\right]$}%
}%
\ifdim\wd\matricesbox>\halfwidth\myboxwidth=\hsize\else\myboxwidth=\halfwidth\fi
\vbox{%
\ifdim\myboxwidth=\hsize
\setbox\onelinebox=\hbox{%
\vbox{\hbox{%
$\Pi_{15,15}$ spans $L_{16.13}$%
}\hbox{%
$322222\slashthree2222232|2\rtimes D_{2}$%
}%
}%
\hfill\copy\matricesbox
}%
\ifdim\wd\onelinebox>\myboxwidth
\hbox to \myboxwidth{%
$\Pi_{15,15}$ spans $L_{16.13}$%
\hfil
$322222\slashthree2222232|2\rtimes D_{2}$%
}%
\box\matricesbox
\else
\hbox to \myboxwidth{%
\unhbox\onelinebox
}%
\fi
\else
\hbox to \myboxwidth{%
$\Pi_{15,15}$ spans $L_{16.13}$%
\hfil}%
\hbox to \myboxwidth{%
$322222\slashthree2222232|2\rtimes D_{2}$%
\hfil}%
\box\matricesbox
\fi
}%
\hfill\discretionary{}{}{}%
\setbox\matricesbox=\hbox{%
{$\left[\!\llap{\phantom{%
\begingroup \smaller\smaller\smaller\begin{tabular}{@{}c@{}}%
\phantom{0}\\\phantom{0}\\\phantom{0}
\end{tabular}\endgroup%
}}\right.$}%
\begingroup \smaller\smaller\smaller\begin{tabular}{@{}c@{}}%
-1\\\phantom{0}\\\phantom{0}
\end{tabular}\endgroup%
\kern3pt%
\begingroup \smaller\smaller\smaller\begin{tabular}{@{}c@{}}%
\phantom{0}\\15/2\\\phantom{0}
\end{tabular}\endgroup%
\kern3pt%
\begingroup \smaller\smaller\smaller\begin{tabular}{@{}c@{}}%
\phantom{0}\\\phantom{0}\\45/2
\end{tabular}\endgroup%
{$\left.\llap{\phantom{%
\begingroup \smaller\smaller\smaller\begin{tabular}{@{}c@{}}%
\phantom{0}\\\phantom{0}\\\phantom{0}
\end{tabular}\endgroup%
}}\!\right]$}%
{$\left[\!\llap{\phantom{%
\begingroup \smaller\smaller\smaller\begin{tabular}{@{}c@{}}%
0\\0\\0
\end{tabular}\endgroup%
}}\right.$}%
\begingroup \smaller\smaller\smaller\begin{tabular}{@{}c@{}}%
30\\11\\1
\end{tabular}\endgroup%
\kern3pt%
\begingroup \smaller\smaller\smaller\begin{tabular}{@{}c@{}}%
9\\3\\1
\end{tabular}\endgroup%
\kern3pt%
\begingroup \smaller\smaller\smaller\begin{tabular}{@{}c@{}}%
5\\1\\1
\end{tabular}\endgroup%
\kern3pt%
\begingroup \smaller\smaller\smaller\begin{tabular}{@{}c@{}}%
9\\0\\2
\end{tabular}\endgroup%
\kern3pt%
\begingroup \smaller\smaller\smaller\begin{tabular}{@{}c@{}}%
5\\-1\\1
\end{tabular}\endgroup%
\kern3pt%
\begingroup \smaller\smaller\smaller\begin{tabular}{@{}c@{}}%
90\\-27\\11
\end{tabular}\endgroup%
\kern3pt%
\begingroup \smaller\smaller\smaller\begin{tabular}{@{}c@{}}%
90\\-30\\8
\end{tabular}\endgroup%
\kern3pt%
\begingroup \smaller\smaller\smaller\begin{tabular}{@{}c@{}}%
5\\-2\\0
\end{tabular}\endgroup%
{$\left.\llap{\phantom{%
\begingroup \smaller\smaller\smaller\begin{tabular}{@{}c@{}}%
0\\0\\0
\end{tabular}\endgroup%
}}\!\right]$}%
}%
\ifdim\wd\matricesbox>\halfwidth\myboxwidth=\hsize\else\myboxwidth=\halfwidth\fi
\vbox{%
\ifdim\myboxwidth=\hsize
\setbox\onelinebox=\hbox{%
\vbox{\hbox{%
$\Pi_{15,16}$ spans $L_{16.13}$%
}\hbox{%
$\slashthree2222232|2322222\rtimes D_{2}$%
}%
}%
\hfill\copy\matricesbox
}%
\ifdim\wd\onelinebox>\myboxwidth
\hbox to \myboxwidth{%
$\Pi_{15,16}$ spans $L_{16.13}$%
\hfil
$\slashthree2222232|2322222\rtimes D_{2}$%
}%
\box\matricesbox
\else
\hbox to \myboxwidth{%
\unhbox\onelinebox
}%
\fi
\else
\hbox to \myboxwidth{%
$\Pi_{15,16}$ spans $L_{16.13}$%
\hfil}%
\hbox to \myboxwidth{%
$\slashthree2222232|2322222\rtimes D_{2}$%
\hfil}%
\box\matricesbox
\fi
}%
\hfill\discretionary{}{}{}%
\setbox\matricesbox=\hbox{%
{$\left[\!\llap{\phantom{%
\begingroup \smaller\smaller\smaller
\endgroup%
}}\!\right]$}%
}%
\ifdim\wd\matricesbox>\halfwidth\myboxwidth=\hsize\else\myboxwidth=\halfwidth\fi
\vbox{%
\ifdim\myboxwidth=\hsize
\setbox\onelinebox=\hbox{%
\vbox{\hbox{%
$\Pi_{15,17}$ spans $L_{16.13}$%
}\hbox{%
$222222223223632$%
}%
}%
\hfill\copy\matricesbox
}%
\ifdim\wd\onelinebox>\myboxwidth
\hbox to \myboxwidth{%
$\Pi_{15,17}$ spans $L_{16.13}$%
\hfil
$222222223223632$%
}%
\box\matricesbox
\else
\hbox to \myboxwidth{%
\unhbox\onelinebox
}%
\fi
\else
\hbox to \myboxwidth{%
$\Pi_{15,17}$ spans $L_{16.13}$%
\hfil}%
\hbox to \myboxwidth{%
$222222223223632$%
\hfil}%
\box\matricesbox
\fi
}%
\hfill\discretionary{}{}{}%
\setbox\matricesbox=\hbox{%
{$\left[\!\llap{\phantom{%
\begingroup \smaller\smaller\smaller
\endgroup%
}}\!\right]$}%
}%
\ifdim\wd\matricesbox>\halfwidth\myboxwidth=\hsize\else\myboxwidth=\halfwidth\fi
\vbox{%
\ifdim\myboxwidth=\hsize
\setbox\onelinebox=\hbox{%
\vbox{\hbox{%
$\Pi_{15,18}$ spans $L_{16.13}$%
}\hbox{%
$222222223632232$%
}%
}%
\hfill\copy\matricesbox
}%
\ifdim\wd\onelinebox>\myboxwidth
\hbox to \myboxwidth{%
$\Pi_{15,18}$ spans $L_{16.13}$%
\hfil
$222222223632232$%
}%
\box\matricesbox
\else
\hbox to \myboxwidth{%
\unhbox\onelinebox
}%
\fi
\else
\hbox to \myboxwidth{%
$\Pi_{15,18}$ spans $L_{16.13}$%
\hfil}%
\hbox to \myboxwidth{%
$222222223632232$%
\hfil}%
\box\matricesbox
\fi
}%
\hfill\discretionary{}{}{}%
\setbox\matricesbox=\hbox{%
{$\left[\!\llap{\phantom{%
\begingroup \smaller\smaller\smaller
\endgroup%
}}\!\right]$}%
}%
\ifdim\wd\matricesbox>\halfwidth\myboxwidth=\hsize\else\myboxwidth=\halfwidth\fi
\vbox{%
\ifdim\myboxwidth=\hsize
\setbox\onelinebox=\hbox{%
\vbox{\hbox{%
$\Pi_{15,19}$ spans $L_{16.13}$%
}\hbox{%
$222222232223632$%
}%
}%
\hfill\copy\matricesbox
}%
\ifdim\wd\onelinebox>\myboxwidth
\hbox to \myboxwidth{%
$\Pi_{15,19}$ spans $L_{16.13}$%
\hfil
$222222232223632$%
}%
\box\matricesbox
\else
\hbox to \myboxwidth{%
\unhbox\onelinebox
}%
\fi
\else
\hbox to \myboxwidth{%
$\Pi_{15,19}$ spans $L_{16.13}$%
\hfil}%
\hbox to \myboxwidth{%
$222222232223632$%
\hfil}%
\box\matricesbox
\fi
}%
\hfill\discretionary{}{}{}%
\setbox\matricesbox=\hbox{%
{$\left[\!\llap{\phantom{%
\begingroup \smaller\smaller\smaller
\endgroup%
}}\!\right]$}%
}%
\ifdim\wd\matricesbox>\halfwidth\myboxwidth=\hsize\else\myboxwidth=\halfwidth\fi
\vbox{%
\ifdim\myboxwidth=\hsize
\setbox\onelinebox=\hbox{%
\vbox{\hbox{%
$\Pi_{15,20}$ spans $L_{16.13}$%
}\hbox{%
$222222322223632$%
}%
}%
\hfill\copy\matricesbox
}%
\ifdim\wd\onelinebox>\myboxwidth
\hbox to \myboxwidth{%
$\Pi_{15,20}$ spans $L_{16.13}$%
\hfil
$222222322223632$%
}%
\box\matricesbox
\else
\hbox to \myboxwidth{%
\unhbox\onelinebox
}%
\fi
\else
\hbox to \myboxwidth{%
$\Pi_{15,20}$ spans $L_{16.13}$%
\hfil}%
\hbox to \myboxwidth{%
$222222322223632$%
\hfil}%
\box\matricesbox
\fi
}%
\hfill\discretionary{}{}{}%
\setbox\matricesbox=\hbox{%
{$\left[\!\llap{\phantom{%
\begingroup \smaller\smaller\smaller
\endgroup%
}}\!\right]$}%
}%
\ifdim\wd\matricesbox>\halfwidth\myboxwidth=\hsize\else\myboxwidth=\halfwidth\fi
\vbox{%
\ifdim\myboxwidth=\hsize
\setbox\onelinebox=\hbox{%
\vbox{\hbox{%
$\Pi_{15,21}$ spans $L_{16.13}$%
}\hbox{%
$222222322232232$%
}%
}%
\hfill\copy\matricesbox
}%
\ifdim\wd\onelinebox>\myboxwidth
\hbox to \myboxwidth{%
$\Pi_{15,21}$ spans $L_{16.13}$%
\hfil
$222222322232232$%
}%
\box\matricesbox
\else
\hbox to \myboxwidth{%
\unhbox\onelinebox
}%
\fi
\else
\hbox to \myboxwidth{%
$\Pi_{15,21}$ spans $L_{16.13}$%
\hfil}%
\hbox to \myboxwidth{%
$222222322232232$%
\hfil}%
\box\matricesbox
\fi
}%
\hfill\discretionary{}{}{}%
\setbox\matricesbox=\hbox{%
{$\left[\!\llap{\phantom{%
\begingroup \smaller\smaller\smaller
\endgroup%
}}\!\right]$}%
}%
\ifdim\wd\matricesbox>\halfwidth\myboxwidth=\hsize\else\myboxwidth=\halfwidth\fi
\vbox{%
\ifdim\myboxwidth=\hsize
\setbox\onelinebox=\hbox{%
\vbox{\hbox{%
$\Pi_{15,22}$ spans $L_{16.13}$%
}\hbox{%
$222222322236322$%
}%
}%
\hfill\copy\matricesbox
}%
\ifdim\wd\onelinebox>\myboxwidth
\hbox to \myboxwidth{%
$\Pi_{15,22}$ spans $L_{16.13}$%
\hfil
$222222322236322$%
}%
\box\matricesbox
\else
\hbox to \myboxwidth{%
\unhbox\onelinebox
}%
\fi
\else
\hbox to \myboxwidth{%
$\Pi_{15,22}$ spans $L_{16.13}$%
\hfil}%
\hbox to \myboxwidth{%
$222222322236322$%
\hfil}%
\box\matricesbox
\fi
}%
\hfill\discretionary{}{}{}%
\setbox\matricesbox=\hbox{%
{$\left[\!\llap{\phantom{%
\begingroup \smaller\smaller\smaller
\endgroup%
}}\!\right]$}%
}%
\ifdim\wd\matricesbox>\halfwidth\myboxwidth=\hsize\else\myboxwidth=\halfwidth\fi
\vbox{%
\ifdim\myboxwidth=\hsize
\setbox\onelinebox=\hbox{%
\vbox{\hbox{%
$\Pi_{15,23}$ spans $L_{16.13}$%
}\hbox{%
$222222322322232$%
}%
}%
\hfill\copy\matricesbox
}%
\ifdim\wd\onelinebox>\myboxwidth
\hbox to \myboxwidth{%
$\Pi_{15,23}$ spans $L_{16.13}$%
\hfil
$222222322322232$%
}%
\box\matricesbox
\else
\hbox to \myboxwidth{%
\unhbox\onelinebox
}%
\fi
\else
\hbox to \myboxwidth{%
$\Pi_{15,23}$ spans $L_{16.13}$%
\hfil}%
\hbox to \myboxwidth{%
$222222322322232$%
\hfil}%
\box\matricesbox
\fi
}%
\hfill\discretionary{}{}{}%
\setbox\matricesbox=\hbox{%
{$\left[\!\llap{\phantom{%
\begingroup \smaller\smaller\smaller
\endgroup%
}}\!\right]$}%
}%
\ifdim\wd\matricesbox>\halfwidth\myboxwidth=\hsize\else\myboxwidth=\halfwidth\fi
\vbox{%
\ifdim\myboxwidth=\hsize
\setbox\onelinebox=\hbox{%
\vbox{\hbox{%
$\Pi_{15,24}$ spans $L_{16.13}$%
}\hbox{%
$222222363222232$%
}%
}%
\hfill\copy\matricesbox
}%
\ifdim\wd\onelinebox>\myboxwidth
\hbox to \myboxwidth{%
$\Pi_{15,24}$ spans $L_{16.13}$%
\hfil
$222222363222232$%
}%
\box\matricesbox
\else
\hbox to \myboxwidth{%
\unhbox\onelinebox
}%
\fi
\else
\hbox to \myboxwidth{%
$\Pi_{15,24}$ spans $L_{16.13}$%
\hfil}%
\hbox to \myboxwidth{%
$222222363222232$%
\hfil}%
\box\matricesbox
\fi
}%
\hfill\discretionary{}{}{}%
\setbox\matricesbox=\hbox{%
{$\left[\!\llap{\phantom{%
\begingroup \smaller\smaller\smaller
\endgroup%
}}\!\right]$}%
}%
\ifdim\wd\matricesbox>\halfwidth\myboxwidth=\hsize\else\myboxwidth=\halfwidth\fi
\vbox{%
\ifdim\myboxwidth=\hsize
\setbox\onelinebox=\hbox{%
\vbox{\hbox{%
$\Pi_{15,25}$ spans $L_{16.13}$%
}\hbox{%
$222223222223632$%
}%
}%
\hfill\copy\matricesbox
}%
\ifdim\wd\onelinebox>\myboxwidth
\hbox to \myboxwidth{%
$\Pi_{15,25}$ spans $L_{16.13}$%
\hfil
$222223222223632$%
}%
\box\matricesbox
\else
\hbox to \myboxwidth{%
\unhbox\onelinebox
}%
\fi
\else
\hbox to \myboxwidth{%
$\Pi_{15,25}$ spans $L_{16.13}$%
\hfil}%
\hbox to \myboxwidth{%
$222223222223632$%
\hfil}%
\box\matricesbox
\fi
}%
\hfill\discretionary{}{}{}%
\setbox\matricesbox=\hbox{%
{$\left[\!\llap{\phantom{%
\begingroup \smaller\smaller\smaller
\endgroup%
}}\!\right]$}%
}%
\ifdim\wd\matricesbox>\halfwidth\myboxwidth=\hsize\else\myboxwidth=\halfwidth\fi
\vbox{%
\ifdim\myboxwidth=\hsize
\setbox\onelinebox=\hbox{%
\vbox{\hbox{%
$\Pi_{15,26}$ spans $L_{16.13}$%
}\hbox{%
$222223222232232$%
}%
}%
\hfill\copy\matricesbox
}%
\ifdim\wd\onelinebox>\myboxwidth
\hbox to \myboxwidth{%
$\Pi_{15,26}$ spans $L_{16.13}$%
\hfil
$222223222232232$%
}%
\box\matricesbox
\else
\hbox to \myboxwidth{%
\unhbox\onelinebox
}%
\fi
\else
\hbox to \myboxwidth{%
$\Pi_{15,26}$ spans $L_{16.13}$%
\hfil}%
\hbox to \myboxwidth{%
$222223222232232$%
\hfil}%
\box\matricesbox
\fi
}%
\hfill\discretionary{}{}{}%
\setbox\matricesbox=\hbox{%
{$\left[\!\llap{\phantom{%
\begingroup \smaller\smaller\smaller
\endgroup%
}}\!\right]$}%
}%
\ifdim\wd\matricesbox>\halfwidth\myboxwidth=\hsize\else\myboxwidth=\halfwidth\fi
\vbox{%
\ifdim\myboxwidth=\hsize
\setbox\onelinebox=\hbox{%
\vbox{\hbox{%
$\Pi_{15,27}$ spans $L_{31.7}$%
}\hbox{%
$223222222322322$%
}%
}%
\hfill\copy\matricesbox
}%
\ifdim\wd\onelinebox>\myboxwidth
\hbox to \myboxwidth{%
$\Pi_{15,27}$ spans $L_{31.7}$%
\hfil
$223222222322322$%
}%
\box\matricesbox
\else
\hbox to \myboxwidth{%
\unhbox\onelinebox
}%
\fi
\else
\hbox to \myboxwidth{%
$\Pi_{15,27}$ spans $L_{31.7}$%
\hfil}%
\hbox to \myboxwidth{%
$223222222322322$%
\hfil}%
\box\matricesbox
\fi
}%
\hfill\discretionary{}{}{}%
\setbox\matricesbox=\hbox{%
{$\left[\!\llap{\phantom{%
\begingroup \smaller\smaller\smaller
\endgroup%
}}\!\right]$}%
}%
\ifdim\wd\matricesbox>\halfwidth\myboxwidth=\hsize\else\myboxwidth=\halfwidth\fi
\vbox{%
\ifdim\myboxwidth=\hsize
\setbox\onelinebox=\hbox{%
\vbox{\hbox{%
$\Pi_{15,28}$ spans $L_{22.4}$%
}\hbox{%
$222232222322322$%
}%
}%
\hfill\copy\matricesbox
}%
\ifdim\wd\onelinebox>\myboxwidth
\hbox to \myboxwidth{%
$\Pi_{15,28}$ spans $L_{22.4}$%
\hfil
$222232222322322$%
}%
\box\matricesbox
\else
\hbox to \myboxwidth{%
\unhbox\onelinebox
}%
\fi
\else
\hbox to \myboxwidth{%
$\Pi_{15,28}$ spans $L_{22.4}$%
\hfil}%
\hbox to \myboxwidth{%
$222232222322322$%
\hfil}%
\box\matricesbox
\fi
}%
\hfill\discretionary{}{}{}%
\setbox\matricesbox=\hbox{%
{$\left[\!\llap{\phantom{%
\begingroup \smaller\smaller\smaller
\endgroup%
}}\!\right]$}%
}%
\ifdim\wd\matricesbox>\halfwidth\myboxwidth=\hsize\else\myboxwidth=\halfwidth\fi
\vbox{%
\ifdim\myboxwidth=\hsize
\setbox\onelinebox=\hbox{%
\vbox{\hbox{%
$\Pi_{15,29}$ spans $L_{16.13}$%
}\hbox{%
$363222222322222$%
}%
}%
\hfill\copy\matricesbox
}%
\ifdim\wd\onelinebox>\myboxwidth
\hbox to \myboxwidth{%
$\Pi_{15,29}$ spans $L_{16.13}$%
\hfil
$363222222322222$%
}%
\box\matricesbox
\else
\hbox to \myboxwidth{%
\unhbox\onelinebox
}%
\fi
\else
\hbox to \myboxwidth{%
$\Pi_{15,29}$ spans $L_{16.13}$%
\hfil}%
\hbox to \myboxwidth{%
$363222222322222$%
\hfil}%
\box\matricesbox
\fi
}%
\hfill\discretionary{}{}{}%
\setbox\matricesbox=\hbox{%
{$\left[\!\llap{\phantom{%
\begingroup \smaller\smaller\smaller
\endgroup%
}}\!\right]$}%
}%
\ifdim\wd\matricesbox>\halfwidth\myboxwidth=\hsize\else\myboxwidth=\halfwidth\fi
\vbox{%
\ifdim\myboxwidth=\hsize
\setbox\onelinebox=\hbox{%
\vbox{\hbox{%
$\Pi_{15,30}$ spans $L_{142.20}$%
}\hbox{%
$\infty422\infty22\infty224\infty422$%
}%
}%
\hfill\copy\matricesbox
}%
\ifdim\wd\onelinebox>\myboxwidth
\hbox to \myboxwidth{%
$\Pi_{15,30}$ spans $L_{142.20}$%
\hfil
$\infty422\infty22\infty224\infty422$%
}%
\box\matricesbox
\else
\hbox to \myboxwidth{%
\unhbox\onelinebox
}%
\fi
\else
\hbox to \myboxwidth{%
$\Pi_{15,30}$ spans $L_{142.20}$%
\hfil}%
\hbox to \myboxwidth{%
$\infty422\infty22\infty224\infty422$%
\hfil}%
\box\matricesbox
\fi
}%
\hfill\discretionary{}{}{}%
\setbox\matricesbox=\hbox{%
{$\left[\!\llap{\phantom{%
\begingroup \smaller\smaller\smaller
\endgroup%
}}\!\right]$}%
}%
\ifdim\wd\matricesbox>\halfwidth\myboxwidth=\hsize\else\myboxwidth=\halfwidth\fi
\vbox{%
\ifdim\myboxwidth=\hsize
\setbox\onelinebox=\hbox{%
\vbox{\hbox{%
$\Pi_{15,31}$ spans $L_{142.20}$%
}\hbox{%
$\infty422\infty22\infty422\infty422$%
}%
}%
\hfill\copy\matricesbox
}%
\ifdim\wd\onelinebox>\myboxwidth
\hbox to \myboxwidth{%
$\Pi_{15,31}$ spans $L_{142.20}$%
\hfil
$\infty422\infty22\infty422\infty422$%
}%
\box\matricesbox
\else
\hbox to \myboxwidth{%
\unhbox\onelinebox
}%
\fi
\else
\hbox to \myboxwidth{%
$\Pi_{15,31}$ spans $L_{142.20}$%
\hfil}%
\hbox to \myboxwidth{%
$\infty422\infty22\infty422\infty422$%
\hfil}%
\box\matricesbox
\fi
}%
\hfill\discretionary{}{}{}%
\setbox\matricesbox=\hbox{%
{$\left[\!\llap{\phantom{%
\begingroup \smaller\smaller\smaller
\endgroup%
}}\!\right]$}%
}%
\ifdim\wd\matricesbox>\halfwidth\myboxwidth=\hsize\else\myboxwidth=\halfwidth\fi
\vbox{%
\ifdim\myboxwidth=\hsize
\setbox\onelinebox=\hbox{%
\vbox{\hbox{%
$\Pi_{15,32}$ spans $L_{142.20}$%
}\hbox{%
$\infty422\infty22\infty4224\infty22$%
}%
}%
\hfill\copy\matricesbox
}%
\ifdim\wd\onelinebox>\myboxwidth
\hbox to \myboxwidth{%
$\Pi_{15,32}$ spans $L_{142.20}$%
\hfil
$\infty422\infty22\infty4224\infty22$%
}%
\box\matricesbox
\else
\hbox to \myboxwidth{%
\unhbox\onelinebox
}%
\fi
\else
\hbox to \myboxwidth{%
$\Pi_{15,32}$ spans $L_{142.20}$%
\hfil}%
\hbox to \myboxwidth{%
$\infty422\infty22\infty4224\infty22$%
\hfil}%
\box\matricesbox
\fi
}%
\hfill\discretionary{}{}{}%
\setbox\matricesbox=\hbox{%
{$\left[\!\llap{\phantom{%
\begingroup \smaller\smaller\smaller
\endgroup%
}}\!\right]$}%
}%
\ifdim\wd\matricesbox>\halfwidth\myboxwidth=\hsize\else\myboxwidth=\halfwidth\fi
\vbox{%
\ifdim\myboxwidth=\hsize
\setbox\onelinebox=\hbox{%
\vbox{\hbox{%
$\Pi_{15,33}$ spans $L_{142.20}$%
}\hbox{%
$\infty422\infty224\infty22\infty422$%
}%
}%
\hfill\copy\matricesbox
}%
\ifdim\wd\onelinebox>\myboxwidth
\hbox to \myboxwidth{%
$\Pi_{15,33}$ spans $L_{142.20}$%
\hfil
$\infty422\infty224\infty22\infty422$%
}%
\box\matricesbox
\else
\hbox to \myboxwidth{%
\unhbox\onelinebox
}%
\fi
\else
\hbox to \myboxwidth{%
$\Pi_{15,33}$ spans $L_{142.20}$%
\hfil}%
\hbox to \myboxwidth{%
$\infty422\infty224\infty22\infty422$%
\hfil}%
\box\matricesbox
\fi
}%
\hfill\discretionary{}{}{}%
\setbox\matricesbox=\hbox{%
{$\left[\!\llap{\phantom{%
\begingroup \smaller\smaller\smaller
\endgroup%
}}\!\right]$}%
}%
\ifdim\wd\matricesbox>\halfwidth\myboxwidth=\hsize\else\myboxwidth=\halfwidth\fi
\vbox{%
\ifdim\myboxwidth=\hsize
\setbox\onelinebox=\hbox{%
\vbox{\hbox{%
$\Pi_{15,34}$ spans $L_{142.20}$%
}\hbox{%
$\infty422\infty224\infty422\infty22$%
}%
}%
\hfill\copy\matricesbox
}%
\ifdim\wd\onelinebox>\myboxwidth
\hbox to \myboxwidth{%
$\Pi_{15,34}$ spans $L_{142.20}$%
\hfil
$\infty422\infty224\infty422\infty22$%
}%
\box\matricesbox
\else
\hbox to \myboxwidth{%
\unhbox\onelinebox
}%
\fi
\else
\hbox to \myboxwidth{%
$\Pi_{15,34}$ spans $L_{142.20}$%
\hfil}%
\hbox to \myboxwidth{%
$\infty422\infty224\infty422\infty22$%
\hfil}%
\box\matricesbox
\fi
}%
\hfill\discretionary{}{}{}%
\setbox\matricesbox=\hbox{%
{$\left[\!\llap{\phantom{%
\begingroup \smaller\smaller\smaller
\endgroup%
}}\!\right]$}%
}%
\ifdim\wd\matricesbox>\halfwidth\myboxwidth=\hsize\else\myboxwidth=\halfwidth\fi
\vbox{%
\ifdim\myboxwidth=\hsize
\setbox\onelinebox=\hbox{%
\vbox{\hbox{%
$\Pi_{15,35}$ spans $L_{142.20}$%
}\hbox{%
$\infty422\infty4224\infty22\infty22$%
}%
}%
\hfill\copy\matricesbox
}%
\ifdim\wd\onelinebox>\myboxwidth
\hbox to \myboxwidth{%
$\Pi_{15,35}$ spans $L_{142.20}$%
\hfil
$\infty422\infty4224\infty22\infty22$%
}%
\box\matricesbox
\else
\hbox to \myboxwidth{%
\unhbox\onelinebox
}%
\fi
\else
\hbox to \myboxwidth{%
$\Pi_{15,35}$ spans $L_{142.20}$%
\hfil}%
\hbox to \myboxwidth{%
$\infty422\infty4224\infty22\infty22$%
\hfil}%
\box\matricesbox
\fi
}%
\hfill\discretionary{}{}{}%
\setbox\matricesbox=\hbox{%
{$\left[\!\llap{\phantom{%
\begingroup \smaller\smaller\smaller
\endgroup%
}}\!\right]$}%
}%
\ifdim\wd\matricesbox>\halfwidth\myboxwidth=\hsize\else\myboxwidth=\halfwidth\fi
\vbox{%
\ifdim\myboxwidth=\hsize
\setbox\onelinebox=\hbox{%
\vbox{\hbox{%
$\Pi_{15,36}$ spans $L_{142.20}$%
}\hbox{%
$\infty4224\infty422\infty22\infty22$%
}%
}%
\hfill\copy\matricesbox
}%
\ifdim\wd\onelinebox>\myboxwidth
\hbox to \myboxwidth{%
$\Pi_{15,36}$ spans $L_{142.20}$%
\hfil
$\infty4224\infty422\infty22\infty22$%
}%
\box\matricesbox
\else
\hbox to \myboxwidth{%
\unhbox\onelinebox
}%
\fi
\else
\hbox to \myboxwidth{%
$\Pi_{15,36}$ spans $L_{142.20}$%
\hfil}%
\hbox to \myboxwidth{%
$\infty4224\infty422\infty22\infty22$%
\hfil}%
\box\matricesbox
\fi
}%
\hfill\discretionary{}{}{}%
\setbox\matricesbox=\hbox{%
{$\left[\!\llap{\phantom{%
\begingroup \smaller\smaller\smaller
\endgroup%
}}\!\right]$}%
}%
\ifdim\wd\matricesbox>\halfwidth\myboxwidth=\hsize\else\myboxwidth=\halfwidth\fi
\vbox{%
\ifdim\myboxwidth=\hsize
\setbox\onelinebox=\hbox{%
\vbox{\hbox{%
$\Pi_{15,37}$ spans $L_{16.13}$%
}\hbox{%
$322222232222322$%
}%
}%
\hfill\copy\matricesbox
}%
\ifdim\wd\onelinebox>\myboxwidth
\hbox to \myboxwidth{%
$\Pi_{15,37}$ spans $L_{16.13}$%
\hfil
$322222232222322$%
}%
\box\matricesbox
\else
\hbox to \myboxwidth{%
\unhbox\onelinebox
}%
\fi
\else
\hbox to \myboxwidth{%
$\Pi_{15,37}$ spans $L_{16.13}$%
\hfil}%
\hbox to \myboxwidth{%
$322222232222322$%
\hfil}%
\box\matricesbox
\fi
}%
\hfill\discretionary{}{}{}%
\setbox\matricesbox=\hbox{%
{$\left[\!\llap{\phantom{%
\begingroup \smaller\smaller\smaller
\endgroup%
}}\!\right]$}%
}%
\ifdim\wd\matricesbox>\halfwidth\myboxwidth=\hsize\else\myboxwidth=\halfwidth\fi
\vbox{%
\ifdim\myboxwidth=\hsize
\setbox\onelinebox=\hbox{%
\vbox{\hbox{%
$\Pi_{15,38}$ spans $L_{16.13}$%
}\hbox{%
$322222322223222$%
}%
}%
\hfill\copy\matricesbox
}%
\ifdim\wd\onelinebox>\myboxwidth
\hbox to \myboxwidth{%
$\Pi_{15,38}$ spans $L_{16.13}$%
\hfil
$322222322223222$%
}%
\box\matricesbox
\else
\hbox to \myboxwidth{%
\unhbox\onelinebox
}%
\fi
\else
\hbox to \myboxwidth{%
$\Pi_{15,38}$ spans $L_{16.13}$%
\hfil}%
\hbox to \myboxwidth{%
$322222322223222$%
\hfil}%
\box\matricesbox
\fi
}%
\hfill\discretionary{}{}{}%
\setbox\matricesbox=\hbox{%
{$\left[\!\llap{\phantom{%
\begingroup \smaller\smaller\smaller
\endgroup%
}}\!\right]$}%
}%
\ifdim\wd\matricesbox>\halfwidth\myboxwidth=\hsize\else\myboxwidth=\halfwidth\fi
\vbox{%
\ifdim\myboxwidth=\hsize
\setbox\onelinebox=\hbox{%
\vbox{\hbox{%
$\Pi_{15,39}$ spans $L_{16.13}$%
}\hbox{%
$322222322232222$%
}%
}%
\hfill\copy\matricesbox
}%
\ifdim\wd\onelinebox>\myboxwidth
\hbox to \myboxwidth{%
$\Pi_{15,39}$ spans $L_{16.13}$%
\hfil
$322222322232222$%
}%
\box\matricesbox
\else
\hbox to \myboxwidth{%
\unhbox\onelinebox
}%
\fi
\else
\hbox to \myboxwidth{%
$\Pi_{15,39}$ spans $L_{16.13}$%
\hfil}%
\hbox to \myboxwidth{%
$322222322232222$%
\hfil}%
\box\matricesbox
\fi
}%
\hfill\discretionary{}{}{}%

\vskip2pt\hrule\vskip2pt

\leavevmode\setbox\matricesbox=\hbox{%
{$\left[\!\llap{\phantom{%
\begingroup \smaller\smaller\smaller\begin{tabular}{@{}c@{}}%
\phantom{0}\\\phantom{0}\\\phantom{0}
\end{tabular}\endgroup%
}}\right.$}%
\begingroup \smaller\smaller\smaller\begin{tabular}{@{}c@{}}%
-1\\\phantom{0}\\\phantom{0}
\end{tabular}\endgroup%
\kern3pt%
\begingroup \smaller\smaller\smaller\begin{tabular}{@{}c@{}}%
\phantom{0}\\30\\\phantom{0}
\end{tabular}\endgroup%
\kern3pt%
\begingroup \smaller\smaller\smaller\begin{tabular}{@{}c@{}}%
\phantom{0}\\\phantom{0}\\30
\end{tabular}\endgroup%
{$\left.\llap{\phantom{%
\begingroup \smaller\smaller\smaller\begin{tabular}{@{}c@{}}%
\phantom{0}\\\phantom{0}\\\phantom{0}
\end{tabular}\endgroup%
}}\!\right]$}%
{$\left[\!\llap{\phantom{%
\begingroup \smaller\smaller\smaller\begin{tabular}{@{}c@{}}%
0\\0\\0
\end{tabular}\endgroup%
}}\right.$}%
\begingroup \smaller\smaller\smaller\begin{tabular}{@{}c@{}}%
5\\0\\1
\end{tabular}\endgroup%
\kern3pt%
\begingroup \smaller\smaller\smaller\begin{tabular}{@{}c@{}}%
24\\-2\\4
\end{tabular}\endgroup%
\kern3pt%
\begingroup \smaller\smaller\smaller\begin{tabular}{@{}c@{}}%
15\\-2\\2
\end{tabular}\endgroup%
{$\left.\llap{\phantom{%
\begingroup \smaller\smaller\smaller\begin{tabular}{@{}c@{}}%
0\\0\\0
\end{tabular}\endgroup%
}}\!\right]$}%
}%
\ifdim\wd\matricesbox>\halfwidth\myboxwidth=\hsize\else\myboxwidth=\halfwidth\fi
\vbox{%
\ifdim\myboxwidth=\hsize
\setbox\onelinebox=\hbox{%
\vbox{\hbox{%
$\Pi_{16,1}$ spans $L_{127.20}$%
}\hbox{%
$|22|22|22|22|22|22|22|22\rtimes D_{8}$%
}%
}%
\hfill\copy\matricesbox
}%
\ifdim\wd\onelinebox>\myboxwidth
\hbox to \myboxwidth{%
$\Pi_{16,1}$ spans $L_{127.20}$%
\hfil
$|22|22|22|22|22|22|22|22\rtimes D_{8}$%
}%
\box\matricesbox
\else
\hbox to \myboxwidth{%
\unhbox\onelinebox
}%
\fi
\else
\hbox to \myboxwidth{%
$\Pi_{16,1}$ spans $L_{127.20}$%
\hfil}%
\hbox to \myboxwidth{%
$|22|22|22|22|22|22|22|22\rtimes D_{8}$%
\hfil}%
\box\matricesbox
\fi
}%
\hfill\discretionary{}{}{}%
\setbox\matricesbox=\hbox{%
{$\left[\!\llap{\phantom{%
\begingroup \smaller\smaller\smaller\begin{tabular}{@{}c@{}}%
\phantom{0}\\\phantom{0}\\\phantom{0}
\end{tabular}\endgroup%
}}\right.$}%
\begingroup \smaller\smaller\smaller\begin{tabular}{@{}c@{}}%
-12\\\phantom{0}\\\phantom{0}
\end{tabular}\endgroup%
\kern3pt%
\begingroup \smaller\smaller\smaller\begin{tabular}{@{}c@{}}%
\phantom{0}\\1\\\phantom{0}
\end{tabular}\endgroup%
\kern3pt%
\begingroup \smaller\smaller\smaller\begin{tabular}{@{}c@{}}%
\phantom{0}\\\phantom{0}\\1
\end{tabular}\endgroup%
{$\left.\llap{\phantom{%
\begingroup \smaller\smaller\smaller\begin{tabular}{@{}c@{}}%
\phantom{0}\\\phantom{0}\\\phantom{0}
\end{tabular}\endgroup%
}}\!\right]$}%
{$\left[\!\llap{\phantom{%
\begingroup \smaller\smaller\smaller\begin{tabular}{@{}c@{}}%
0\\0\\0
\end{tabular}\endgroup%
}}\right.$}%
\begingroup \smaller\smaller\smaller\begin{tabular}{@{}c@{}}%
2\\1\\-7
\end{tabular}\endgroup%
\kern3pt%
\begingroup \smaller\smaller\smaller\begin{tabular}{@{}c@{}}%
1\\2\\-3
\end{tabular}\endgroup%
{$\left.\llap{\phantom{%
\begingroup \smaller\smaller\smaller\begin{tabular}{@{}c@{}}%
0\\0\\0
\end{tabular}\endgroup%
}}\!\right]$}%
}%
\ifdim\wd\matricesbox>\halfwidth\myboxwidth=\hsize\else\myboxwidth=\halfwidth\fi
\vbox{%
\ifdim\myboxwidth=\hsize
\setbox\onelinebox=\hbox{%
\vbox{\hbox{%
$\Pi_{16,2}$ spans $L_{167.1}$%
}\hbox{%
$4\slashtwo4\slashtwo4\slashtwo4\slashtwo4\slashtwo4\slashtwo4\slashtwo4\slashtwo\rtimes D_{8}$%
}%
}%
\hfill\copy\matricesbox
}%
\ifdim\wd\onelinebox>\myboxwidth
\hbox to \myboxwidth{%
$\Pi_{16,2}$ spans $L_{167.1}$%
\hfil
$4\slashtwo4\slashtwo4\slashtwo4\slashtwo4\slashtwo4\slashtwo4\slashtwo4\slashtwo\rtimes D_{8}$%
}%
\box\matricesbox
\else
\hbox to \myboxwidth{%
\unhbox\onelinebox
}%
\fi
\else
\hbox to \myboxwidth{%
$\Pi_{16,2}$ spans $L_{167.1}$%
\hfil}%
\hbox to \myboxwidth{%
$4\slashtwo4\slashtwo4\slashtwo4\slashtwo4\slashtwo4\slashtwo4\slashtwo4\slashtwo\rtimes D_{8}$%
\hfil}%
\box\matricesbox
\fi
}%
\hfill\discretionary{}{}{}%
\setbox\matricesbox=\hbox{%
{$\left[\!\llap{\phantom{%
\begingroup \smaller\smaller\smaller\begin{tabular}{@{}c@{}}%
\phantom{0}\\\phantom{0}\\\phantom{0}
\end{tabular}\endgroup%
}}\right.$}%
\begingroup \smaller\smaller\smaller\begin{tabular}{@{}c@{}}%
-1\\\phantom{0}\\\phantom{0}
\end{tabular}\endgroup%
\kern3pt%
\begingroup \smaller\smaller\smaller\begin{tabular}{@{}c@{}}%
\phantom{0}\\5\\\phantom{0}
\end{tabular}\endgroup%
\kern3pt%
\begingroup \smaller\smaller\smaller\begin{tabular}{@{}c@{}}%
\phantom{0}\\\phantom{0}\\15
\end{tabular}\endgroup%
{$\left.\llap{\phantom{%
\begingroup \smaller\smaller\smaller\begin{tabular}{@{}c@{}}%
\phantom{0}\\\phantom{0}\\\phantom{0}
\end{tabular}\endgroup%
}}\!\right]$}%
{$\left[\!\llap{\phantom{%
\begingroup \smaller\smaller\smaller\begin{tabular}{@{}c@{}}%
0\\0\\0
\end{tabular}\endgroup%
}}\right.$}%
\begingroup \smaller\smaller\smaller\begin{tabular}{@{}c@{}}%
4\\-2\\0
\end{tabular}\endgroup%
\kern3pt%
\begingroup \smaller\smaller\smaller\begin{tabular}{@{}c@{}}%
15\\-6\\-2
\end{tabular}\endgroup%
\kern3pt%
\begingroup \smaller\smaller\smaller\begin{tabular}{@{}c@{}}%
20\\-6\\-4
\end{tabular}\endgroup%
\kern3pt%
\begingroup \smaller\smaller\smaller\begin{tabular}{@{}c@{}}%
20\\-3\\-5
\end{tabular}\endgroup%
\kern3pt%
\begingroup \smaller\smaller\smaller\begin{tabular}{@{}c@{}}%
15\\0\\-4
\end{tabular}\endgroup%
{$\left.\llap{\phantom{%
\begingroup \smaller\smaller\smaller\begin{tabular}{@{}c@{}}%
0\\0\\0
\end{tabular}\endgroup%
}}\!\right]$}%
}%
\ifdim\wd\matricesbox>\halfwidth\myboxwidth=\hsize\else\myboxwidth=\halfwidth\fi
\vbox{%
\ifdim\myboxwidth=\hsize
\setbox\onelinebox=\hbox{%
\vbox{\hbox{%
$\Pi_{16,3}$ spans $L_{31.7}$%
}\hbox{%
$|2232|2322|2232|2322\rtimes D_{4}$%
}%
}%
\hfill\copy\matricesbox
}%
\ifdim\wd\onelinebox>\myboxwidth
\hbox to \myboxwidth{%
$\Pi_{16,3}$ spans $L_{31.7}$%
\hfil
$|2232|2322|2232|2322\rtimes D_{4}$%
}%
\box\matricesbox
\else
\hbox to \myboxwidth{%
\unhbox\onelinebox
}%
\fi
\else
\hbox to \myboxwidth{%
$\Pi_{16,3}$ spans $L_{31.7}$%
\hfil}%
\hbox to \myboxwidth{%
$|2232|2322|2232|2322\rtimes D_{4}$%
\hfil}%
\box\matricesbox
\fi
}%
\hfill\discretionary{}{}{}%
\setbox\matricesbox=\hbox{%
{$\left[\!\llap{\phantom{%
\begingroup \smaller\smaller\smaller\begin{tabular}{@{}c@{}}%
\phantom{0}\\\phantom{0}\\\phantom{0}
\end{tabular}\endgroup%
}}\right.$}%
\begingroup \smaller\smaller\smaller\begin{tabular}{@{}c@{}}%
-1\\\phantom{0}\\\phantom{0}
\end{tabular}\endgroup%
\kern3pt%
\begingroup \smaller\smaller\smaller\begin{tabular}{@{}c@{}}%
\phantom{0}\\18\\\phantom{0}
\end{tabular}\endgroup%
\kern3pt%
\begingroup \smaller\smaller\smaller\begin{tabular}{@{}c@{}}%
\phantom{0}\\\phantom{0}\\18
\end{tabular}\endgroup%
{$\left.\llap{\phantom{%
\begingroup \smaller\smaller\smaller\begin{tabular}{@{}c@{}}%
\phantom{0}\\\phantom{0}\\\phantom{0}
\end{tabular}\endgroup%
}}\!\right]$}%
{$\left[\!\llap{\phantom{%
\begingroup \smaller\smaller\smaller\begin{tabular}{@{}c@{}}%
0\\0\\0
\end{tabular}\endgroup%
}}\right.$}%
\begingroup \smaller\smaller\smaller\begin{tabular}{@{}c@{}}%
8\\2\\0
\end{tabular}\endgroup%
\kern3pt%
\begingroup \smaller\smaller\smaller\begin{tabular}{@{}c@{}}%
72\\16\\-6
\end{tabular}\endgroup%
\kern3pt%
\begingroup \smaller\smaller\smaller\begin{tabular}{@{}c@{}}%
36\\7\\-5
\end{tabular}\endgroup%
\kern3pt%
\begingroup \smaller\smaller\smaller\begin{tabular}{@{}c@{}}%
9\\1\\-2
\end{tabular}\endgroup%
\kern3pt%
\begingroup \smaller\smaller\smaller\begin{tabular}{@{}c@{}}%
8\\0\\-2
\end{tabular}\endgroup%
{$\left.\llap{\phantom{%
\begingroup \smaller\smaller\smaller\begin{tabular}{@{}c@{}}%
0\\0\\0
\end{tabular}\endgroup%
}}\!\right]$}%
}%
\ifdim\wd\matricesbox>\halfwidth\myboxwidth=\hsize\else\myboxwidth=\halfwidth\fi
\vbox{%
\ifdim\myboxwidth=\hsize
\setbox\onelinebox=\hbox{%
\vbox{\hbox{%
$\Pi_{16,4}$ spans $L_{142.20}$%
}\hbox{%
$\infty42|24\infty2|2\infty42|24\infty2|2\rtimes D_{4}$%
}%
}%
\hfill\copy\matricesbox
}%
\ifdim\wd\onelinebox>\myboxwidth
\hbox to \myboxwidth{%
$\Pi_{16,4}$ spans $L_{142.20}$%
\hfil
$\infty42|24\infty2|2\infty42|24\infty2|2\rtimes D_{4}$%
}%
\box\matricesbox
\else
\hbox to \myboxwidth{%
\unhbox\onelinebox
}%
\fi
\else
\hbox to \myboxwidth{%
$\Pi_{16,4}$ spans $L_{142.20}$%
\hfil}%
\hbox to \myboxwidth{%
$\infty42|24\infty2|2\infty42|24\infty2|2\rtimes D_{4}$%
\hfil}%
\box\matricesbox
\fi
}%
\hfill\discretionary{}{}{}%
\setbox\matricesbox=\hbox{%
{$\left[\!\llap{\phantom{%
\begingroup \smaller\smaller\smaller\begin{tabular}{@{}c@{}}%
\phantom{0}\\\phantom{0}\\\phantom{0}
\end{tabular}\endgroup%
}}\right.$}%
\begingroup \smaller\smaller\smaller\begin{tabular}{@{}c@{}}%
-1\\\phantom{0}\\\phantom{0}
\end{tabular}\endgroup%
\kern3pt%
\begingroup \smaller\smaller\smaller\begin{tabular}{@{}c@{}}%
\phantom{0}\\15/2\\\phantom{0}
\end{tabular}\endgroup%
\kern3pt%
\begingroup \smaller\smaller\smaller\begin{tabular}{@{}c@{}}%
\phantom{0}\\\phantom{0}\\45/2
\end{tabular}\endgroup%
{$\left.\llap{\phantom{%
\begingroup \smaller\smaller\smaller\begin{tabular}{@{}c@{}}%
\phantom{0}\\\phantom{0}\\\phantom{0}
\end{tabular}\endgroup%
}}\!\right]$}%
{$\left[\!\llap{\phantom{%
\begingroup \smaller\smaller\smaller\begin{tabular}{@{}c@{}}%
0\\0\\0
\end{tabular}\endgroup%
}}\right.$}%
\begingroup \smaller\smaller\smaller\begin{tabular}{@{}c@{}}%
5\\2\\0
\end{tabular}\endgroup%
\kern3pt%
\begingroup \smaller\smaller\smaller\begin{tabular}{@{}c@{}}%
9\\3\\1
\end{tabular}\endgroup%
\kern3pt%
\begingroup \smaller\smaller\smaller\begin{tabular}{@{}c@{}}%
30\\7\\5
\end{tabular}\endgroup%
\kern3pt%
\begingroup \smaller\smaller\smaller\begin{tabular}{@{}c@{}}%
30\\4\\6
\end{tabular}\endgroup%
\kern3pt%
\begingroup \smaller\smaller\smaller\begin{tabular}{@{}c@{}}%
9\\0\\2
\end{tabular}\endgroup%
{$\left.\llap{\phantom{%
\begingroup \smaller\smaller\smaller\begin{tabular}{@{}c@{}}%
0\\0\\0
\end{tabular}\endgroup%
}}\!\right]$}%
}%
\ifdim\wd\matricesbox>\halfwidth\myboxwidth=\hsize\else\myboxwidth=\halfwidth\fi
\vbox{%
\ifdim\myboxwidth=\hsize
\setbox\onelinebox=\hbox{%
\vbox{\hbox{%
$\Pi_{16,5}$ spans $L_{16.13}$%
}\hbox{%
$322|2232|2322|2232|2\rtimes D_{4}$%
}%
}%
\hfill\copy\matricesbox
}%
\ifdim\wd\onelinebox>\myboxwidth
\hbox to \myboxwidth{%
$\Pi_{16,5}$ spans $L_{16.13}$%
\hfil
$322|2232|2322|2232|2\rtimes D_{4}$%
}%
\box\matricesbox
\else
\hbox to \myboxwidth{%
\unhbox\onelinebox
}%
\fi
\else
\hbox to \myboxwidth{%
$\Pi_{16,5}$ spans $L_{16.13}$%
\hfil}%
\hbox to \myboxwidth{%
$322|2232|2322|2232|2\rtimes D_{4}$%
\hfil}%
\box\matricesbox
\fi
}%
\hfill\discretionary{}{}{}%
\setbox\matricesbox=\hbox{%
{$\left[\!\llap{\phantom{%
\begingroup \smaller\smaller\smaller\begin{tabular}{@{}c@{}}%
\phantom{0}\\\phantom{0}\\\phantom{0}
\end{tabular}\endgroup%
}}\right.$}%
\begingroup \smaller\smaller\smaller\begin{tabular}{@{}c@{}}%
-1\\\phantom{0}\\\phantom{0}
\end{tabular}\endgroup%
\kern3pt%
\begingroup \smaller\smaller\smaller\begin{tabular}{@{}c@{}}%
\phantom{0}\\15/2\\\phantom{0}
\end{tabular}\endgroup%
\kern3pt%
\begingroup \smaller\smaller\smaller\begin{tabular}{@{}c@{}}%
\phantom{0}\\\phantom{0}\\45/2
\end{tabular}\endgroup%
{$\left.\llap{\phantom{%
\begingroup \smaller\smaller\smaller\begin{tabular}{@{}c@{}}%
\phantom{0}\\\phantom{0}\\\phantom{0}
\end{tabular}\endgroup%
}}\!\right]$}%
{$\left[\!\llap{\phantom{%
\begingroup \smaller\smaller\smaller\begin{tabular}{@{}c@{}}%
0\\0\\0
\end{tabular}\endgroup%
}}\right.$}%
\begingroup \smaller\smaller\smaller\begin{tabular}{@{}c@{}}%
30\\11\\1
\end{tabular}\endgroup%
\kern3pt%
\begingroup \smaller\smaller\smaller\begin{tabular}{@{}c@{}}%
9\\3\\1
\end{tabular}\endgroup%
\kern3pt%
\begingroup \smaller\smaller\smaller\begin{tabular}{@{}c@{}}%
5\\1\\1
\end{tabular}\endgroup%
\kern3pt%
\begingroup \smaller\smaller\smaller\begin{tabular}{@{}c@{}}%
90\\3\\19
\end{tabular}\endgroup%
{$\left.\llap{\phantom{%
\begingroup \smaller\smaller\smaller\begin{tabular}{@{}c@{}}%
0\\0\\0
\end{tabular}\endgroup%
}}\!\right]$}%
}%
\ifdim\wd\matricesbox>\halfwidth\myboxwidth=\hsize\else\myboxwidth=\halfwidth\fi
\vbox{%
\ifdim\myboxwidth=\hsize
\setbox\onelinebox=\hbox{%
\vbox{\hbox{%
$\Pi_{16,6}$ spans $L_{16.13}$%
}\hbox{%
$\slashthree222\slashthree222\slashthree222\slashthree222\rtimes D_{4}$%
}%
}%
\hfill\copy\matricesbox
}%
\ifdim\wd\onelinebox>\myboxwidth
\hbox to \myboxwidth{%
$\Pi_{16,6}$ spans $L_{16.13}$%
\hfil
$\slashthree222\slashthree222\slashthree222\slashthree222\rtimes D_{4}$%
}%
\box\matricesbox
\else
\hbox to \myboxwidth{%
\unhbox\onelinebox
}%
\fi
\else
\hbox to \myboxwidth{%
$\Pi_{16,6}$ spans $L_{16.13}$%
\hfil}%
\hbox to \myboxwidth{%
$\slashthree222\slashthree222\slashthree222\slashthree222\rtimes D_{4}$%
\hfil}%
\box\matricesbox
\fi
}%
\hfill\discretionary{}{}{}%
\setbox\matricesbox=\hbox{%
{$\left[\!\llap{\phantom{%
\begingroup \smaller\smaller\smaller\begin{tabular}{@{}c@{}}%
\phantom{0}\\\phantom{0}\\\phantom{0}
\end{tabular}\endgroup%
}}\right.$}%
\begingroup \smaller\smaller\smaller\begin{tabular}{@{}c@{}}%
-1\\\phantom{0}\\\phantom{0}
\end{tabular}\endgroup%
\kern3pt%
\begingroup \smaller\smaller\smaller\begin{tabular}{@{}c@{}}%
\phantom{0}\\21/2\\\phantom{0}
\end{tabular}\endgroup%
\kern3pt%
\begingroup \smaller\smaller\smaller\begin{tabular}{@{}c@{}}%
\phantom{0}\\\phantom{0}\\7/2
\end{tabular}\endgroup%
{$\left.\llap{\phantom{%
\begingroup \smaller\smaller\smaller\begin{tabular}{@{}c@{}}%
\phantom{0}\\\phantom{0}\\\phantom{0}
\end{tabular}\endgroup%
}}\!\right]$}%
{$\left[\!\llap{\phantom{%
\begingroup \smaller\smaller\smaller\begin{tabular}{@{}c@{}}%
0\\0\\0
\end{tabular}\endgroup%
}}\right.$}%
\begingroup \smaller\smaller\smaller\begin{tabular}{@{}c@{}}%
6\\-2\\0
\end{tabular}\endgroup%
\kern3pt%
\begingroup \smaller\smaller\smaller\begin{tabular}{@{}c@{}}%
7\\-2\\2
\end{tabular}\endgroup%
\kern3pt%
\begingroup \smaller\smaller\smaller\begin{tabular}{@{}c@{}}%
42\\-8\\18
\end{tabular}\endgroup%
\kern3pt%
\begingroup \smaller\smaller\smaller\begin{tabular}{@{}c@{}}%
42\\-5\\21
\end{tabular}\endgroup%
\kern3pt%
\begingroup \smaller\smaller\smaller\begin{tabular}{@{}c@{}}%
7\\0\\4
\end{tabular}\endgroup%
{$\left.\llap{\phantom{%
\begingroup \smaller\smaller\smaller\begin{tabular}{@{}c@{}}%
0\\0\\0
\end{tabular}\endgroup%
}}\!\right]$}%
}%
\ifdim\wd\matricesbox>\halfwidth\myboxwidth=\hsize\else\myboxwidth=\halfwidth\fi
\vbox{%
\ifdim\myboxwidth=\hsize
\setbox\onelinebox=\hbox{%
\vbox{\hbox{%
$\Pi_{16,7}$ spans $L_{22.4}$%
}\hbox{%
$322|2232|2322|2232|2\rtimes D_{4}$%
}%
}%
\hfill\copy\matricesbox
}%
\ifdim\wd\onelinebox>\myboxwidth
\hbox to \myboxwidth{%
$\Pi_{16,7}$ spans $L_{22.4}$%
\hfil
$322|2232|2322|2232|2\rtimes D_{4}$%
}%
\box\matricesbox
\else
\hbox to \myboxwidth{%
\unhbox\onelinebox
}%
\fi
\else
\hbox to \myboxwidth{%
$\Pi_{16,7}$ spans $L_{22.4}$%
\hfil}%
\hbox to \myboxwidth{%
$322|2232|2322|2232|2\rtimes D_{4}$%
\hfil}%
\box\matricesbox
\fi
}%
\hfill\discretionary{}{}{}%
\setbox\matricesbox=\hbox{%
{$\left[\!\llap{\phantom{%
\begingroup \smaller\smaller\smaller\begin{tabular}{@{}c@{}}%
\phantom{0}\\\phantom{0}\\\phantom{0}
\end{tabular}\endgroup%
}}\right.$}%
\begingroup \smaller\smaller\smaller\begin{tabular}{@{}c@{}}%
-1\\\phantom{0}\\\phantom{0}
\end{tabular}\endgroup%
\kern3pt%
\begingroup \smaller\smaller\smaller\begin{tabular}{@{}c@{}}%
\phantom{0}\\45/2\\\phantom{0}
\end{tabular}\endgroup%
\kern3pt%
\begingroup \smaller\smaller\smaller\begin{tabular}{@{}c@{}}%
\phantom{0}\\\phantom{0}\\15/2
\end{tabular}\endgroup%
{$\left.\llap{\phantom{%
\begingroup \smaller\smaller\smaller\begin{tabular}{@{}c@{}}%
\phantom{0}\\\phantom{0}\\\phantom{0}
\end{tabular}\endgroup%
}}\!\right]$}%
{$\left[\!\llap{\phantom{%
\begingroup \smaller\smaller\smaller\begin{tabular}{@{}c@{}}%
0\\0\\0
\end{tabular}\endgroup%
}}\right.$}%
\begingroup \smaller\smaller\smaller\begin{tabular}{@{}c@{}}%
9\\2\\0
\end{tabular}\endgroup%
\kern3pt%
\begingroup \smaller\smaller\smaller\begin{tabular}{@{}c@{}}%
5\\1\\-1
\end{tabular}\endgroup%
\kern3pt%
\begingroup \smaller\smaller\smaller\begin{tabular}{@{}c@{}}%
90\\11\\-27
\end{tabular}\endgroup%
\kern3pt%
\begingroup \smaller\smaller\smaller\begin{tabular}{@{}c@{}}%
90\\8\\-30
\end{tabular}\endgroup%
\kern3pt%
\begingroup \smaller\smaller\smaller\begin{tabular}{@{}c@{}}%
5\\0\\-2
\end{tabular}\endgroup%
{$\left.\llap{\phantom{%
\begingroup \smaller\smaller\smaller\begin{tabular}{@{}c@{}}%
0\\0\\0
\end{tabular}\endgroup%
}}\!\right]$}%
}%
\ifdim\wd\matricesbox>\halfwidth\myboxwidth=\hsize\else\myboxwidth=\halfwidth\fi
\vbox{%
\ifdim\myboxwidth=\hsize
\setbox\onelinebox=\hbox{%
\vbox{\hbox{%
$\Pi_{16,8}$ spans $L_{16.13}$%
}\hbox{%
$322|2232|2322|2232|2\rtimes D_{4}$%
}%
}%
\hfill\copy\matricesbox
}%
\ifdim\wd\onelinebox>\myboxwidth
\hbox to \myboxwidth{%
$\Pi_{16,8}$ spans $L_{16.13}$%
\hfil
$322|2232|2322|2232|2\rtimes D_{4}$%
}%
\box\matricesbox
\else
\hbox to \myboxwidth{%
\unhbox\onelinebox
}%
\fi
\else
\hbox to \myboxwidth{%
$\Pi_{16,8}$ spans $L_{16.13}$%
\hfil}%
\hbox to \myboxwidth{%
$322|2232|2322|2232|2\rtimes D_{4}$%
\hfil}%
\box\matricesbox
\fi
}%
\hfill\discretionary{}{}{}%
\setbox\matricesbox=\hbox{%
{$\left[\!\llap{\phantom{%
\begingroup \smaller\smaller\smaller\begin{tabular}{@{}c@{}}%
\phantom{0}\\\phantom{0}\\\phantom{0}
\end{tabular}\endgroup%
}}\right.$}%
\begingroup \smaller\smaller\smaller\begin{tabular}{@{}c@{}}%
-1\\\phantom{0}\\\phantom{0}
\end{tabular}\endgroup%
\kern3pt%
\begingroup \smaller\smaller\smaller\begin{tabular}{@{}c@{}}%
\phantom{0}\\9\\\phantom{0}
\end{tabular}\endgroup%
\kern3pt%
\begingroup \smaller\smaller\smaller\begin{tabular}{@{}c@{}}%
\phantom{0}\\\phantom{0}\\9
\end{tabular}\endgroup%
{$\left.\llap{\phantom{%
\begingroup \smaller\smaller\smaller\begin{tabular}{@{}c@{}}%
\phantom{0}\\\phantom{0}\\\phantom{0}
\end{tabular}\endgroup%
}}\!\right]$}%
{$\left[\!\llap{\phantom{%
\begingroup \smaller\smaller\smaller\begin{tabular}{@{}c@{}}%
0\\0\\0
\end{tabular}\endgroup%
}}\right.$}%
\begingroup \smaller\smaller\smaller\begin{tabular}{@{}c@{}}%
9\\3\\-1
\end{tabular}\endgroup%
\kern3pt%
\begingroup \smaller\smaller\smaller\begin{tabular}{@{}c@{}}%
8\\2\\-2
\end{tabular}\endgroup%
\kern3pt%
\begingroup \smaller\smaller\smaller\begin{tabular}{@{}c@{}}%
72\\10\\-22
\end{tabular}\endgroup%
\kern3pt%
\begingroup \smaller\smaller\smaller\begin{tabular}{@{}c@{}}%
36\\2\\-12
\end{tabular}\endgroup%
{$\left.\llap{\phantom{%
\begingroup \smaller\smaller\smaller\begin{tabular}{@{}c@{}}%
0\\0\\0
\end{tabular}\endgroup%
}}\!\right]$}%
}%
\ifdim\wd\matricesbox>\halfwidth\myboxwidth=\hsize\else\myboxwidth=\halfwidth\fi
\vbox{%
\ifdim\myboxwidth=\hsize
\setbox\onelinebox=\hbox{%
\vbox{\hbox{%
$\Pi_{16,9}$ spans $L_{142.20}$%
}\hbox{%
$422\slashinfty224\slashinfty422\slashinfty224\slashinfty\rtimes D_{4}$%
}%
}%
\hfill\copy\matricesbox
}%
\ifdim\wd\onelinebox>\myboxwidth
\hbox to \myboxwidth{%
$\Pi_{16,9}$ spans $L_{142.20}$%
\hfil
$422\slashinfty224\slashinfty422\slashinfty224\slashinfty\rtimes D_{4}$%
}%
\box\matricesbox
\else
\hbox to \myboxwidth{%
\unhbox\onelinebox
}%
\fi
\else
\hbox to \myboxwidth{%
$\Pi_{16,9}$ spans $L_{142.20}$%
\hfil}%
\hbox to \myboxwidth{%
$422\slashinfty224\slashinfty422\slashinfty224\slashinfty\rtimes D_{4}$%
\hfil}%
\box\matricesbox
\fi
}%
\hfill\discretionary{}{}{}%
\setbox\matricesbox=\hbox{%
{$\left[\!\llap{\phantom{%
\begingroup \smaller\smaller\smaller\begin{tabular}{@{}c@{}}%
\phantom{0}\\\phantom{0}\\\phantom{0}
\end{tabular}\endgroup%
}}\right.$}%
\begingroup \smaller\smaller\smaller\begin{tabular}{@{}c@{}}%
-1\\\phantom{0}\\\phantom{0}
\end{tabular}\endgroup%
\kern3pt%
\begingroup \smaller\smaller\smaller\begin{tabular}{@{}c@{}}%
\phantom{0}\\18\\\phantom{0}
\end{tabular}\endgroup%
\kern3pt%
\begingroup \smaller\smaller\smaller\begin{tabular}{@{}c@{}}%
\phantom{0}\\\phantom{0}\\18
\end{tabular}\endgroup%
{$\left.\llap{\phantom{%
\begingroup \smaller\smaller\smaller\begin{tabular}{@{}c@{}}%
\phantom{0}\\\phantom{0}\\\phantom{0}
\end{tabular}\endgroup%
}}\!\right]$}%
{$\left[\!\llap{\phantom{%
\begingroup \smaller\smaller\smaller\begin{tabular}{@{}c@{}}%
0\\0\\0
\end{tabular}\endgroup%
}}\right.$}%
\begingroup \smaller\smaller\smaller\begin{tabular}{@{}c@{}}%
9\\1\\-2
\end{tabular}\endgroup%
\kern3pt%
\begingroup \smaller\smaller\smaller\begin{tabular}{@{}c@{}}%
36\\7\\-5
\end{tabular}\endgroup%
\kern3pt%
\begingroup \smaller\smaller\smaller\begin{tabular}{@{}c@{}}%
72\\16\\-6
\end{tabular}\endgroup%
\kern3pt%
\begingroup \smaller\smaller\smaller\begin{tabular}{@{}c@{}}%
8\\2\\0
\end{tabular}\endgroup%
{$\left.\llap{\phantom{%
\begingroup \smaller\smaller\smaller\begin{tabular}{@{}c@{}}%
0\\0\\0
\end{tabular}\endgroup%
}}\!\right]$}%
}%
\ifdim\wd\matricesbox>\halfwidth\myboxwidth=\hsize\else\myboxwidth=\halfwidth\fi
\vbox{%
\ifdim\myboxwidth=\hsize
\setbox\onelinebox=\hbox{%
\vbox{\hbox{%
$\Pi_{16,10}$ spans $L_{142.20}$%
}\hbox{%
$\infty422\infty422\infty422\infty422\rtimes C_{4}$%
}%
}%
\hfill\copy\matricesbox
}%
\ifdim\wd\onelinebox>\myboxwidth
\hbox to \myboxwidth{%
$\Pi_{16,10}$ spans $L_{142.20}$%
\hfil
$\infty422\infty422\infty422\infty422\rtimes C_{4}$%
}%
\box\matricesbox
\else
\hbox to \myboxwidth{%
\unhbox\onelinebox
}%
\fi
\else
\hbox to \myboxwidth{%
$\Pi_{16,10}$ spans $L_{142.20}$%
\hfil}%
\hbox to \myboxwidth{%
$\infty422\infty422\infty422\infty422\rtimes C_{4}$%
\hfil}%
\box\matricesbox
\fi
}%
\hfill\discretionary{}{}{}%
\setbox\matricesbox=\hbox{%
{$\left[\!\llap{\phantom{%
\begingroup \smaller\smaller\smaller\begin{tabular}{@{}c@{}}%
\phantom{0}\\\phantom{0}\\\phantom{0}
\end{tabular}\endgroup%
}}\right.$}%
\begingroup \smaller\smaller\smaller\begin{tabular}{@{}c@{}}%
-1\\\phantom{0}\\\phantom{0}
\end{tabular}\endgroup%
\kern3pt%
\begingroup \smaller\smaller\smaller\begin{tabular}{@{}c@{}}%
\phantom{0}\\45/2\\\phantom{0}
\end{tabular}\endgroup%
\kern3pt%
\begingroup \smaller\smaller\smaller\begin{tabular}{@{}c@{}}%
\phantom{0}\\\phantom{0}\\15/2
\end{tabular}\endgroup%
{$\left.\llap{\phantom{%
\begingroup \smaller\smaller\smaller\begin{tabular}{@{}c@{}}%
\phantom{0}\\\phantom{0}\\\phantom{0}
\end{tabular}\endgroup%
}}\!\right]$}%
{$\left[\!\llap{\phantom{%
\begingroup \smaller\smaller\smaller\begin{tabular}{@{}c@{}}%
0\\0\\0
\end{tabular}\endgroup%
}}\right.$}%
\begingroup \smaller\smaller\smaller\begin{tabular}{@{}c@{}}%
9\\2\\0
\end{tabular}\endgroup%
\kern3pt%
\begingroup \smaller\smaller\smaller\begin{tabular}{@{}c@{}}%
5\\1\\1
\end{tabular}\endgroup%
\kern3pt%
\begingroup \smaller\smaller\smaller\begin{tabular}{@{}c@{}}%
9\\1\\3
\end{tabular}\endgroup%
\kern3pt%
\begingroup \smaller\smaller\smaller\begin{tabular}{@{}c@{}}%
5\\0\\2
\end{tabular}\endgroup%
\kern3pt%
\begingroup \smaller\smaller\smaller\begin{tabular}{@{}c@{}}%
90\\-8\\30
\end{tabular}\endgroup%
\kern3pt%
\begingroup \smaller\smaller\smaller\begin{tabular}{@{}c@{}}%
90\\-11\\27
\end{tabular}\endgroup%
\kern3pt%
\begingroup \smaller\smaller\smaller\begin{tabular}{@{}c@{}}%
30\\-5\\7
\end{tabular}\endgroup%
\kern3pt%
\begingroup \smaller\smaller\smaller\begin{tabular}{@{}c@{}}%
30\\-6\\4
\end{tabular}\endgroup%
\kern3pt%
\begingroup \smaller\smaller\smaller\begin{tabular}{@{}c@{}}%
9\\-2\\0
\end{tabular}\endgroup%
{$\left.\llap{\phantom{%
\begingroup \smaller\smaller\smaller\begin{tabular}{@{}c@{}}%
0\\0\\0
\end{tabular}\endgroup%
}}\!\right]$}%
}%
\ifdim\wd\matricesbox>\halfwidth\myboxwidth=\hsize\else\myboxwidth=\halfwidth\fi
\vbox{%
\ifdim\myboxwidth=\hsize
\setbox\onelinebox=\hbox{%
\vbox{\hbox{%
$\Pi_{16,11}$ spans $L_{16.13}$%
}\hbox{%
$222|22223632|23632\rtimes D_{2}$%
}%
}%
\hfill\copy\matricesbox
}%
\ifdim\wd\onelinebox>\myboxwidth
\hbox to \myboxwidth{%
$\Pi_{16,11}$ spans $L_{16.13}$%
\hfil
$222|22223632|23632\rtimes D_{2}$%
}%
\box\matricesbox
\else
\hbox to \myboxwidth{%
\unhbox\onelinebox
}%
\fi
\else
\hbox to \myboxwidth{%
$\Pi_{16,11}$ spans $L_{16.13}$%
\hfil}%
\hbox to \myboxwidth{%
$222|22223632|23632\rtimes D_{2}$%
\hfil}%
\box\matricesbox
\fi
}%
\hfill\discretionary{}{}{}%
\setbox\matricesbox=\hbox{%
{$\left[\!\llap{\phantom{%
\begingroup \smaller\smaller\smaller\begin{tabular}{@{}c@{}}%
\phantom{0}\\\phantom{0}\\\phantom{0}
\end{tabular}\endgroup%
}}\right.$}%
\begingroup \smaller\smaller\smaller\begin{tabular}{@{}c@{}}%
-1\\\phantom{0}\\\phantom{0}
\end{tabular}\endgroup%
\kern3pt%
\begingroup \smaller\smaller\smaller\begin{tabular}{@{}c@{}}%
\phantom{0}\\15/2\\\phantom{0}
\end{tabular}\endgroup%
\kern3pt%
\begingroup \smaller\smaller\smaller\begin{tabular}{@{}c@{}}%
\phantom{0}\\\phantom{0}\\45/2
\end{tabular}\endgroup%
{$\left.\llap{\phantom{%
\begingroup \smaller\smaller\smaller\begin{tabular}{@{}c@{}}%
\phantom{0}\\\phantom{0}\\\phantom{0}
\end{tabular}\endgroup%
}}\!\right]$}%
{$\left[\!\llap{\phantom{%
\begingroup \smaller\smaller\smaller\begin{tabular}{@{}c@{}}%
0\\0\\0
\end{tabular}\endgroup%
}}\right.$}%
\begingroup \smaller\smaller\smaller\begin{tabular}{@{}c@{}}%
5\\-2\\0
\end{tabular}\endgroup%
\kern3pt%
\begingroup \smaller\smaller\smaller\begin{tabular}{@{}c@{}}%
9\\-3\\1
\end{tabular}\endgroup%
\kern3pt%
\begingroup \smaller\smaller\smaller\begin{tabular}{@{}c@{}}%
5\\-1\\1
\end{tabular}\endgroup%
\kern3pt%
\begingroup \smaller\smaller\smaller\begin{tabular}{@{}c@{}}%
9\\0\\2
\end{tabular}\endgroup%
\kern3pt%
\begingroup \smaller\smaller\smaller\begin{tabular}{@{}c@{}}%
30\\4\\6
\end{tabular}\endgroup%
\kern3pt%
\begingroup \smaller\smaller\smaller\begin{tabular}{@{}c@{}}%
30\\7\\5
\end{tabular}\endgroup%
\kern3pt%
\begingroup \smaller\smaller\smaller\begin{tabular}{@{}c@{}}%
90\\27\\11
\end{tabular}\endgroup%
\kern3pt%
\begingroup \smaller\smaller\smaller\begin{tabular}{@{}c@{}}%
90\\30\\8
\end{tabular}\endgroup%
\kern3pt%
\begingroup \smaller\smaller\smaller\begin{tabular}{@{}c@{}}%
5\\2\\0
\end{tabular}\endgroup%
{$\left.\llap{\phantom{%
\begingroup \smaller\smaller\smaller\begin{tabular}{@{}c@{}}%
0\\0\\0
\end{tabular}\endgroup%
}}\!\right]$}%
}%
\ifdim\wd\matricesbox>\halfwidth\myboxwidth=\hsize\else\myboxwidth=\halfwidth\fi
\vbox{%
\ifdim\myboxwidth=\hsize
\setbox\onelinebox=\hbox{%
\vbox{\hbox{%
$\Pi_{16,12}$ spans $L_{16.13}$%
}\hbox{%
$22|22223632|236322\rtimes D_{2}$%
}%
}%
\hfill\copy\matricesbox
}%
\ifdim\wd\onelinebox>\myboxwidth
\hbox to \myboxwidth{%
$\Pi_{16,12}$ spans $L_{16.13}$%
\hfil
$22|22223632|236322\rtimes D_{2}$%
}%
\box\matricesbox
\else
\hbox to \myboxwidth{%
\unhbox\onelinebox
}%
\fi
\else
\hbox to \myboxwidth{%
$\Pi_{16,12}$ spans $L_{16.13}$%
\hfil}%
\hbox to \myboxwidth{%
$22|22223632|236322\rtimes D_{2}$%
\hfil}%
\box\matricesbox
\fi
}%
\hfill\discretionary{}{}{}%
\setbox\matricesbox=\hbox{%
{$\left[\!\llap{\phantom{%
\begingroup \smaller\smaller\smaller\begin{tabular}{@{}c@{}}%
\phantom{0}\\\phantom{0}\\\phantom{0}
\end{tabular}\endgroup%
}}\right.$}%
\begingroup \smaller\smaller\smaller\begin{tabular}{@{}c@{}}%
-1\\\phantom{0}\\\phantom{0}
\end{tabular}\endgroup%
\kern3pt%
\begingroup \smaller\smaller\smaller\begin{tabular}{@{}c@{}}%
\phantom{0}\\15/2\\\phantom{0}
\end{tabular}\endgroup%
\kern3pt%
\begingroup \smaller\smaller\smaller\begin{tabular}{@{}c@{}}%
\phantom{0}\\\phantom{0}\\45/2
\end{tabular}\endgroup%
{$\left.\llap{\phantom{%
\begingroup \smaller\smaller\smaller\begin{tabular}{@{}c@{}}%
\phantom{0}\\\phantom{0}\\\phantom{0}
\end{tabular}\endgroup%
}}\!\right]$}%
{$\left[\!\llap{\phantom{%
\begingroup \smaller\smaller\smaller\begin{tabular}{@{}c@{}}%
0\\0\\0
\end{tabular}\endgroup%
}}\right.$}%
\begingroup \smaller\smaller\smaller\begin{tabular}{@{}c@{}}%
5\\-2\\0
\end{tabular}\endgroup%
\kern3pt%
\begingroup \smaller\smaller\smaller\begin{tabular}{@{}c@{}}%
9\\-3\\1
\end{tabular}\endgroup%
\kern3pt%
\begingroup \smaller\smaller\smaller\begin{tabular}{@{}c@{}}%
5\\-1\\1
\end{tabular}\endgroup%
\kern3pt%
\begingroup \smaller\smaller\smaller\begin{tabular}{@{}c@{}}%
90\\-3\\19
\end{tabular}\endgroup%
\kern3pt%
\begingroup \smaller\smaller\smaller\begin{tabular}{@{}c@{}}%
90\\3\\19
\end{tabular}\endgroup%
\kern3pt%
\begingroup \smaller\smaller\smaller\begin{tabular}{@{}c@{}}%
5\\1\\1
\end{tabular}\endgroup%
\kern3pt%
\begingroup \smaller\smaller\smaller\begin{tabular}{@{}c@{}}%
90\\27\\11
\end{tabular}\endgroup%
\kern3pt%
\begingroup \smaller\smaller\smaller\begin{tabular}{@{}c@{}}%
90\\30\\8
\end{tabular}\endgroup%
\kern3pt%
\begingroup \smaller\smaller\smaller\begin{tabular}{@{}c@{}}%
5\\2\\0
\end{tabular}\endgroup%
{$\left.\llap{\phantom{%
\begingroup \smaller\smaller\smaller\begin{tabular}{@{}c@{}}%
0\\0\\0
\end{tabular}\endgroup%
}}\!\right]$}%
}%
\ifdim\wd\matricesbox>\halfwidth\myboxwidth=\hsize\else\myboxwidth=\halfwidth\fi
\vbox{%
\ifdim\myboxwidth=\hsize
\setbox\onelinebox=\hbox{%
\vbox{\hbox{%
$\Pi_{16,13}$ spans $L_{16.13}$%
}\hbox{%
$22|22232232|232232\rtimes D_{2}$%
}%
}%
\hfill\copy\matricesbox
}%
\ifdim\wd\onelinebox>\myboxwidth
\hbox to \myboxwidth{%
$\Pi_{16,13}$ spans $L_{16.13}$%
\hfil
$22|22232232|232232\rtimes D_{2}$%
}%
\box\matricesbox
\else
\hbox to \myboxwidth{%
\unhbox\onelinebox
}%
\fi
\else
\hbox to \myboxwidth{%
$\Pi_{16,13}$ spans $L_{16.13}$%
\hfil}%
\hbox to \myboxwidth{%
$22|22232232|232232\rtimes D_{2}$%
\hfil}%
\box\matricesbox
\fi
}%
\hfill\discretionary{}{}{}%
\setbox\matricesbox=\hbox{%
{$\left[\!\llap{\phantom{%
\begingroup \smaller\smaller\smaller\begin{tabular}{@{}c@{}}%
\phantom{0}\\\phantom{0}\\\phantom{0}
\end{tabular}\endgroup%
}}\right.$}%
\begingroup \smaller\smaller\smaller\begin{tabular}{@{}c@{}}%
-1\\\phantom{0}\\\phantom{0}
\end{tabular}\endgroup%
\kern3pt%
\begingroup \smaller\smaller\smaller\begin{tabular}{@{}c@{}}%
\phantom{0}\\15/2\\\phantom{0}
\end{tabular}\endgroup%
\kern3pt%
\begingroup \smaller\smaller\smaller\begin{tabular}{@{}c@{}}%
\phantom{0}\\\phantom{0}\\45/2
\end{tabular}\endgroup%
{$\left.\llap{\phantom{%
\begingroup \smaller\smaller\smaller\begin{tabular}{@{}c@{}}%
\phantom{0}\\\phantom{0}\\\phantom{0}
\end{tabular}\endgroup%
}}\!\right]$}%
{$\left[\!\llap{\phantom{%
\begingroup \smaller\smaller\smaller\begin{tabular}{@{}c@{}}%
0\\0\\0
\end{tabular}\endgroup%
}}\right.$}%
\begingroup \smaller\smaller\smaller\begin{tabular}{@{}c@{}}%
5\\-2\\0
\end{tabular}\endgroup%
\kern3pt%
\begingroup \smaller\smaller\smaller\begin{tabular}{@{}c@{}}%
9\\-3\\1
\end{tabular}\endgroup%
\kern3pt%
\begingroup \smaller\smaller\smaller\begin{tabular}{@{}c@{}}%
5\\-1\\1
\end{tabular}\endgroup%
\kern3pt%
\begingroup \smaller\smaller\smaller\begin{tabular}{@{}c@{}}%
90\\-3\\19
\end{tabular}\endgroup%
\kern3pt%
\begingroup \smaller\smaller\smaller\begin{tabular}{@{}c@{}}%
90\\3\\19
\end{tabular}\endgroup%
\kern3pt%
\begingroup \smaller\smaller\smaller\begin{tabular}{@{}c@{}}%
30\\4\\6
\end{tabular}\endgroup%
\kern3pt%
\begingroup \smaller\smaller\smaller\begin{tabular}{@{}c@{}}%
30\\7\\5
\end{tabular}\endgroup%
\kern3pt%
\begingroup \smaller\smaller\smaller\begin{tabular}{@{}c@{}}%
9\\3\\1
\end{tabular}\endgroup%
\kern3pt%
\begingroup \smaller\smaller\smaller\begin{tabular}{@{}c@{}}%
5\\2\\0
\end{tabular}\endgroup%
{$\left.\llap{\phantom{%
\begingroup \smaller\smaller\smaller\begin{tabular}{@{}c@{}}%
0\\0\\0
\end{tabular}\endgroup%
}}\!\right]$}%
}%
\ifdim\wd\matricesbox>\halfwidth\myboxwidth=\hsize\else\myboxwidth=\halfwidth\fi
\vbox{%
\ifdim\myboxwidth=\hsize
\setbox\onelinebox=\hbox{%
\vbox{\hbox{%
$\Pi_{16,14}$ spans $L_{16.13}$%
}\hbox{%
$22|22236322|223632\rtimes D_{2}$%
}%
}%
\hfill\copy\matricesbox
}%
\ifdim\wd\onelinebox>\myboxwidth
\hbox to \myboxwidth{%
$\Pi_{16,14}$ spans $L_{16.13}$%
\hfil
$22|22236322|223632\rtimes D_{2}$%
}%
\box\matricesbox
\else
\hbox to \myboxwidth{%
\unhbox\onelinebox
}%
\fi
\else
\hbox to \myboxwidth{%
$\Pi_{16,14}$ spans $L_{16.13}$%
\hfil}%
\hbox to \myboxwidth{%
$22|22236322|223632\rtimes D_{2}$%
\hfil}%
\box\matricesbox
\fi
}%
\hfill\discretionary{}{}{}%
\setbox\matricesbox=\hbox{%
{$\left[\!\llap{\phantom{%
\begingroup \smaller\smaller\smaller\begin{tabular}{@{}c@{}}%
\phantom{0}\\\phantom{0}\\\phantom{0}
\end{tabular}\endgroup%
}}\right.$}%
\begingroup \smaller\smaller\smaller\begin{tabular}{@{}c@{}}%
-1\\\phantom{0}\\\phantom{0}
\end{tabular}\endgroup%
\kern3pt%
\begingroup \smaller\smaller\smaller\begin{tabular}{@{}c@{}}%
\phantom{0}\\5\\\phantom{0}
\end{tabular}\endgroup%
\kern3pt%
\begingroup \smaller\smaller\smaller\begin{tabular}{@{}c@{}}%
\phantom{0}\\\phantom{0}\\15
\end{tabular}\endgroup%
{$\left.\llap{\phantom{%
\begingroup \smaller\smaller\smaller\begin{tabular}{@{}c@{}}%
\phantom{0}\\\phantom{0}\\\phantom{0}
\end{tabular}\endgroup%
}}\!\right]$}%
{$\left[\!\llap{\phantom{%
\begingroup \smaller\smaller\smaller\begin{tabular}{@{}c@{}}%
0\\0\\0
\end{tabular}\endgroup%
}}\right.$}%
\begingroup \smaller\smaller\smaller\begin{tabular}{@{}c@{}}%
20\\9\\1
\end{tabular}\endgroup%
\kern3pt%
\begingroup \smaller\smaller\smaller\begin{tabular}{@{}c@{}}%
15\\6\\2
\end{tabular}\endgroup%
\kern3pt%
\begingroup \smaller\smaller\smaller\begin{tabular}{@{}c@{}}%
4\\1\\1
\end{tabular}\endgroup%
\kern3pt%
\begingroup \smaller\smaller\smaller\begin{tabular}{@{}c@{}}%
15\\0\\4
\end{tabular}\endgroup%
\kern3pt%
\begingroup \smaller\smaller\smaller\begin{tabular}{@{}c@{}}%
20\\-3\\5
\end{tabular}\endgroup%
\kern3pt%
\begingroup \smaller\smaller\smaller\begin{tabular}{@{}c@{}}%
20\\-6\\4
\end{tabular}\endgroup%
\kern3pt%
\begingroup \smaller\smaller\smaller\begin{tabular}{@{}c@{}}%
15\\-6\\2
\end{tabular}\endgroup%
\kern3pt%
\begingroup \smaller\smaller\smaller\begin{tabular}{@{}c@{}}%
20\\-9\\1
\end{tabular}\endgroup%
{$\left.\llap{\phantom{%
\begingroup \smaller\smaller\smaller\begin{tabular}{@{}c@{}}%
0\\0\\0
\end{tabular}\endgroup%
}}\!\right]$}%
}%
\ifdim\wd\matricesbox>\halfwidth\myboxwidth=\hsize\else\myboxwidth=\halfwidth\fi
\vbox{%
\ifdim\myboxwidth=\hsize
\setbox\onelinebox=\hbox{%
\vbox{\hbox{%
$\Pi_{16,15}$ spans $L_{31.7}$%
}\hbox{%
$22\slashthree2222322\slashthree22322\rtimes D_{2}$%
}%
}%
\hfill\copy\matricesbox
}%
\ifdim\wd\onelinebox>\myboxwidth
\hbox to \myboxwidth{%
$\Pi_{16,15}$ spans $L_{31.7}$%
\hfil
$22\slashthree2222322\slashthree22322\rtimes D_{2}$%
}%
\box\matricesbox
\else
\hbox to \myboxwidth{%
\unhbox\onelinebox
}%
\fi
\else
\hbox to \myboxwidth{%
$\Pi_{16,15}$ spans $L_{31.7}$%
\hfil}%
\hbox to \myboxwidth{%
$22\slashthree2222322\slashthree22322\rtimes D_{2}$%
\hfil}%
\box\matricesbox
\fi
}%
\hfill\discretionary{}{}{}%
\setbox\matricesbox=\hbox{%
{$\left[\!\llap{\phantom{%
\begingroup \smaller\smaller\smaller\begin{tabular}{@{}c@{}}%
\phantom{0}\\\phantom{0}\\\phantom{0}
\end{tabular}\endgroup%
}}\right.$}%
\begingroup \smaller\smaller\smaller\begin{tabular}{@{}c@{}}%
-1\\\phantom{0}\\\phantom{0}
\end{tabular}\endgroup%
\kern3pt%
\begingroup \smaller\smaller\smaller\begin{tabular}{@{}c@{}}%
\phantom{0}\\15\\\phantom{0}
\end{tabular}\endgroup%
\kern3pt%
\begingroup \smaller\smaller\smaller\begin{tabular}{@{}c@{}}%
\phantom{0}\\\phantom{0}\\5
\end{tabular}\endgroup%
{$\left.\llap{\phantom{%
\begingroup \smaller\smaller\smaller\begin{tabular}{@{}c@{}}%
\phantom{0}\\\phantom{0}\\\phantom{0}
\end{tabular}\endgroup%
}}\!\right]$}%
{$\left[\!\llap{\phantom{%
\begingroup \smaller\smaller\smaller\begin{tabular}{@{}c@{}}%
0\\0\\0
\end{tabular}\endgroup%
}}\right.$}%
\begingroup \smaller\smaller\smaller\begin{tabular}{@{}c@{}}%
15\\4\\0
\end{tabular}\endgroup%
\kern3pt%
\begingroup \smaller\smaller\smaller\begin{tabular}{@{}c@{}}%
20\\5\\3
\end{tabular}\endgroup%
\kern3pt%
\begingroup \smaller\smaller\smaller\begin{tabular}{@{}c@{}}%
20\\4\\6
\end{tabular}\endgroup%
\kern3pt%
\begingroup \smaller\smaller\smaller\begin{tabular}{@{}c@{}}%
15\\2\\6
\end{tabular}\endgroup%
\kern3pt%
\begingroup \smaller\smaller\smaller\begin{tabular}{@{}c@{}}%
20\\1\\9
\end{tabular}\endgroup%
\kern3pt%
\begingroup \smaller\smaller\smaller\begin{tabular}{@{}c@{}}%
20\\-1\\9
\end{tabular}\endgroup%
\kern3pt%
\begingroup \smaller\smaller\smaller\begin{tabular}{@{}c@{}}%
15\\-2\\6
\end{tabular}\endgroup%
\kern3pt%
\begingroup \smaller\smaller\smaller\begin{tabular}{@{}c@{}}%
4\\-1\\1
\end{tabular}\endgroup%
\kern3pt%
\begingroup \smaller\smaller\smaller\begin{tabular}{@{}c@{}}%
15\\-4\\0
\end{tabular}\endgroup%
{$\left.\llap{\phantom{%
\begingroup \smaller\smaller\smaller\begin{tabular}{@{}c@{}}%
0\\0\\0
\end{tabular}\endgroup%
}}\!\right]$}%
}%
\ifdim\wd\matricesbox>\halfwidth\myboxwidth=\hsize\else\myboxwidth=\halfwidth\fi
\vbox{%
\ifdim\myboxwidth=\hsize
\setbox\onelinebox=\hbox{%
\vbox{\hbox{%
$\Pi_{16,16}$ spans $L_{31.7}$%
}\hbox{%
$2232232|23223222|2\rtimes D_{2}$%
}%
}%
\hfill\copy\matricesbox
}%
\ifdim\wd\onelinebox>\myboxwidth
\hbox to \myboxwidth{%
$\Pi_{16,16}$ spans $L_{31.7}$%
\hfil
$2232232|23223222|2\rtimes D_{2}$%
}%
\box\matricesbox
\else
\hbox to \myboxwidth{%
\unhbox\onelinebox
}%
\fi
\else
\hbox to \myboxwidth{%
$\Pi_{16,16}$ spans $L_{31.7}$%
\hfil}%
\hbox to \myboxwidth{%
$2232232|23223222|2\rtimes D_{2}$%
\hfil}%
\box\matricesbox
\fi
}%
\hfill\discretionary{}{}{}%
\setbox\matricesbox=\hbox{%
{$\left[\!\llap{\phantom{%
\begingroup \smaller\smaller\smaller\begin{tabular}{@{}c@{}}%
\phantom{0}\\\phantom{0}\\\phantom{0}
\end{tabular}\endgroup%
}}\right.$}%
\begingroup \smaller\smaller\smaller\begin{tabular}{@{}c@{}}%
-1\\\phantom{0}\\\phantom{0}
\end{tabular}\endgroup%
\kern3pt%
\begingroup \smaller\smaller\smaller\begin{tabular}{@{}c@{}}%
\phantom{0}\\7/2\\\phantom{0}
\end{tabular}\endgroup%
\kern3pt%
\begingroup \smaller\smaller\smaller\begin{tabular}{@{}c@{}}%
\phantom{0}\\\phantom{0}\\21/2
\end{tabular}\endgroup%
{$\left.\llap{\phantom{%
\begingroup \smaller\smaller\smaller\begin{tabular}{@{}c@{}}%
\phantom{0}\\\phantom{0}\\\phantom{0}
\end{tabular}\endgroup%
}}\!\right]$}%
{$\left[\!\llap{\phantom{%
\begingroup \smaller\smaller\smaller\begin{tabular}{@{}c@{}}%
0\\0\\0
\end{tabular}\endgroup%
}}\right.$}%
\begingroup \smaller\smaller\smaller\begin{tabular}{@{}c@{}}%
7\\-4\\0
\end{tabular}\endgroup%
\kern3pt%
\begingroup \smaller\smaller\smaller\begin{tabular}{@{}c@{}}%
6\\-3\\-1
\end{tabular}\endgroup%
\kern3pt%
\begingroup \smaller\smaller\smaller\begin{tabular}{@{}c@{}}%
7\\-2\\-2
\end{tabular}\endgroup%
\kern3pt%
\begingroup \smaller\smaller\smaller\begin{tabular}{@{}c@{}}%
42\\-3\\-13
\end{tabular}\endgroup%
\kern3pt%
\begingroup \smaller\smaller\smaller\begin{tabular}{@{}c@{}}%
42\\3\\-13
\end{tabular}\endgroup%
\kern3pt%
\begingroup \smaller\smaller\smaller\begin{tabular}{@{}c@{}}%
7\\2\\-2
\end{tabular}\endgroup%
\kern3pt%
\begingroup \smaller\smaller\smaller\begin{tabular}{@{}c@{}}%
42\\18\\-8
\end{tabular}\endgroup%
\kern3pt%
\begingroup \smaller\smaller\smaller\begin{tabular}{@{}c@{}}%
42\\21\\-5
\end{tabular}\endgroup%
\kern3pt%
\begingroup \smaller\smaller\smaller\begin{tabular}{@{}c@{}}%
7\\4\\0
\end{tabular}\endgroup%
{$\left.\llap{\phantom{%
\begingroup \smaller\smaller\smaller\begin{tabular}{@{}c@{}}%
0\\0\\0
\end{tabular}\endgroup%
}}\!\right]$}%
}%
\ifdim\wd\matricesbox>\halfwidth\myboxwidth=\hsize\else\myboxwidth=\halfwidth\fi
\vbox{%
\ifdim\myboxwidth=\hsize
\setbox\onelinebox=\hbox{%
\vbox{\hbox{%
$\Pi_{16,17}$ spans $L_{22.4}$%
}\hbox{%
$2|22232232|2322322\rtimes D_{2}$%
}%
}%
\hfill\copy\matricesbox
}%
\ifdim\wd\onelinebox>\myboxwidth
\hbox to \myboxwidth{%
$\Pi_{16,17}$ spans $L_{22.4}$%
\hfil
$2|22232232|2322322\rtimes D_{2}$%
}%
\box\matricesbox
\else
\hbox to \myboxwidth{%
\unhbox\onelinebox
}%
\fi
\else
\hbox to \myboxwidth{%
$\Pi_{16,17}$ spans $L_{22.4}$%
\hfil}%
\hbox to \myboxwidth{%
$2|22232232|2322322\rtimes D_{2}$%
\hfil}%
\box\matricesbox
\fi
}%
\hfill\discretionary{}{}{}%
\setbox\matricesbox=\hbox{%
{$\left[\!\llap{\phantom{%
\begingroup \smaller\smaller\smaller\begin{tabular}{@{}c@{}}%
\phantom{0}\\\phantom{0}\\\phantom{0}
\end{tabular}\endgroup%
}}\right.$}%
\begingroup \smaller\smaller\smaller\begin{tabular}{@{}c@{}}%
-1\\\phantom{0}\\\phantom{0}
\end{tabular}\endgroup%
\kern3pt%
\begingroup \smaller\smaller\smaller\begin{tabular}{@{}c@{}}%
\phantom{0}\\45/2\\\phantom{0}
\end{tabular}\endgroup%
\kern3pt%
\begingroup \smaller\smaller\smaller\begin{tabular}{@{}c@{}}%
\phantom{0}\\\phantom{0}\\15/2
\end{tabular}\endgroup%
{$\left.\llap{\phantom{%
\begingroup \smaller\smaller\smaller\begin{tabular}{@{}c@{}}%
\phantom{0}\\\phantom{0}\\\phantom{0}
\end{tabular}\endgroup%
}}\!\right]$}%
{$\left[\!\llap{\phantom{%
\begingroup \smaller\smaller\smaller\begin{tabular}{@{}c@{}}%
0\\0\\0
\end{tabular}\endgroup%
}}\right.$}%
\begingroup \smaller\smaller\smaller\begin{tabular}{@{}c@{}}%
9\\2\\0
\end{tabular}\endgroup%
\kern3pt%
\begingroup \smaller\smaller\smaller\begin{tabular}{@{}c@{}}%
5\\1\\1
\end{tabular}\endgroup%
\kern3pt%
\begingroup \smaller\smaller\smaller\begin{tabular}{@{}c@{}}%
9\\1\\3
\end{tabular}\endgroup%
\kern3pt%
\begingroup \smaller\smaller\smaller\begin{tabular}{@{}c@{}}%
30\\1\\11
\end{tabular}\endgroup%
\kern3pt%
\begingroup \smaller\smaller\smaller\begin{tabular}{@{}c@{}}%
30\\-1\\11
\end{tabular}\endgroup%
\kern3pt%
\begingroup \smaller\smaller\smaller\begin{tabular}{@{}c@{}}%
90\\-8\\30
\end{tabular}\endgroup%
\kern3pt%
\begingroup \smaller\smaller\smaller\begin{tabular}{@{}c@{}}%
90\\-11\\27
\end{tabular}\endgroup%
\kern3pt%
\begingroup \smaller\smaller\smaller\begin{tabular}{@{}c@{}}%
5\\-1\\1
\end{tabular}\endgroup%
\kern3pt%
\begingroup \smaller\smaller\smaller\begin{tabular}{@{}c@{}}%
9\\-2\\0
\end{tabular}\endgroup%
{$\left.\llap{\phantom{%
\begingroup \smaller\smaller\smaller\begin{tabular}{@{}c@{}}%
0\\0\\0
\end{tabular}\endgroup%
}}\!\right]$}%
}%
\ifdim\wd\matricesbox>\halfwidth\myboxwidth=\hsize\else\myboxwidth=\halfwidth\fi
\vbox{%
\ifdim\myboxwidth=\hsize
\setbox\onelinebox=\hbox{%
\vbox{\hbox{%
$\Pi_{16,18}$ spans $L_{16.13}$%
}\hbox{%
$363222|22236322|22\rtimes D_{2}$%
}%
}%
\hfill\copy\matricesbox
}%
\ifdim\wd\onelinebox>\myboxwidth
\hbox to \myboxwidth{%
$\Pi_{16,18}$ spans $L_{16.13}$%
\hfil
$363222|22236322|22\rtimes D_{2}$%
}%
\box\matricesbox
\else
\hbox to \myboxwidth{%
\unhbox\onelinebox
}%
\fi
\else
\hbox to \myboxwidth{%
$\Pi_{16,18}$ spans $L_{16.13}$%
\hfil}%
\hbox to \myboxwidth{%
$363222|22236322|22\rtimes D_{2}$%
\hfil}%
\box\matricesbox
\fi
}%
\hfill\discretionary{}{}{}%
\setbox\matricesbox=\hbox{%
{$\left[\!\llap{\phantom{%
\begingroup \smaller\smaller\smaller\begin{tabular}{@{}c@{}}%
\phantom{0}\\\phantom{0}\\\phantom{0}
\end{tabular}\endgroup%
}}\right.$}%
\begingroup \smaller\smaller\smaller\begin{tabular}{@{}c@{}}%
-1\\\phantom{0}\\\phantom{0}
\end{tabular}\endgroup%
\kern3pt%
\begingroup \smaller\smaller\smaller\begin{tabular}{@{}c@{}}%
\phantom{0}\\45/2\\\phantom{0}
\end{tabular}\endgroup%
\kern3pt%
\begingroup \smaller\smaller\smaller\begin{tabular}{@{}c@{}}%
\phantom{0}\\\phantom{0}\\15/2
\end{tabular}\endgroup%
{$\left.\llap{\phantom{%
\begingroup \smaller\smaller\smaller\begin{tabular}{@{}c@{}}%
\phantom{0}\\\phantom{0}\\\phantom{0}
\end{tabular}\endgroup%
}}\!\right]$}%
{$\left[\!\llap{\phantom{%
\begingroup \smaller\smaller\smaller\begin{tabular}{@{}c@{}}%
0\\0\\0
\end{tabular}\endgroup%
}}\right.$}%
\begingroup \smaller\smaller\smaller\begin{tabular}{@{}c@{}}%
90\\19\\3
\end{tabular}\endgroup%
\kern3pt%
\begingroup \smaller\smaller\smaller\begin{tabular}{@{}c@{}}%
30\\6\\4
\end{tabular}\endgroup%
\kern3pt%
\begingroup \smaller\smaller\smaller\begin{tabular}{@{}c@{}}%
30\\5\\7
\end{tabular}\endgroup%
\kern3pt%
\begingroup \smaller\smaller\smaller\begin{tabular}{@{}c@{}}%
9\\1\\3
\end{tabular}\endgroup%
\kern3pt%
\begingroup \smaller\smaller\smaller\begin{tabular}{@{}c@{}}%
5\\0\\2
\end{tabular}\endgroup%
\kern3pt%
\begingroup \smaller\smaller\smaller\begin{tabular}{@{}c@{}}%
9\\-1\\3
\end{tabular}\endgroup%
\kern3pt%
\begingroup \smaller\smaller\smaller\begin{tabular}{@{}c@{}}%
5\\-1\\1
\end{tabular}\endgroup%
\kern3pt%
\begingroup \smaller\smaller\smaller\begin{tabular}{@{}c@{}}%
90\\-19\\3
\end{tabular}\endgroup%
{$\left.\llap{\phantom{%
\begingroup \smaller\smaller\smaller\begin{tabular}{@{}c@{}}%
0\\0\\0
\end{tabular}\endgroup%
}}\!\right]$}%
}%
\ifdim\wd\matricesbox>\halfwidth\myboxwidth=\hsize\else\myboxwidth=\halfwidth\fi
\vbox{%
\ifdim\myboxwidth=\hsize
\setbox\onelinebox=\hbox{%
\vbox{\hbox{%
$\Pi_{16,19}$ spans $L_{16.13}$%
}\hbox{%
$\slashthree6322222\slashthree2222236\rtimes D_{2}$%
}%
}%
\hfill\copy\matricesbox
}%
\ifdim\wd\onelinebox>\myboxwidth
\hbox to \myboxwidth{%
$\Pi_{16,19}$ spans $L_{16.13}$%
\hfil
$\slashthree6322222\slashthree2222236\rtimes D_{2}$%
}%
\box\matricesbox
\else
\hbox to \myboxwidth{%
\unhbox\onelinebox
}%
\fi
\else
\hbox to \myboxwidth{%
$\Pi_{16,19}$ spans $L_{16.13}$%
\hfil}%
\hbox to \myboxwidth{%
$\slashthree6322222\slashthree2222236\rtimes D_{2}$%
\hfil}%
\box\matricesbox
\fi
}%
\hfill\discretionary{}{}{}%
\setbox\matricesbox=\hbox{%
{$\left[\!\llap{\phantom{%
\begingroup \smaller\smaller\smaller\begin{tabular}{@{}c@{}}%
\phantom{0}\\\phantom{0}\\\phantom{0}
\end{tabular}\endgroup%
}}\right.$}%
\begingroup \smaller\smaller\smaller\begin{tabular}{@{}c@{}}%
-1\\\phantom{0}\\\phantom{0}
\end{tabular}\endgroup%
\kern3pt%
\begingroup \smaller\smaller\smaller\begin{tabular}{@{}c@{}}%
\phantom{0}\\15/2\\\phantom{0}
\end{tabular}\endgroup%
\kern3pt%
\begingroup \smaller\smaller\smaller\begin{tabular}{@{}c@{}}%
\phantom{0}\\\phantom{0}\\45/2
\end{tabular}\endgroup%
{$\left.\llap{\phantom{%
\begingroup \smaller\smaller\smaller\begin{tabular}{@{}c@{}}%
\phantom{0}\\\phantom{0}\\\phantom{0}
\end{tabular}\endgroup%
}}\!\right]$}%
{$\left[\!\llap{\phantom{%
\begingroup \smaller\smaller\smaller\begin{tabular}{@{}c@{}}%
0\\0\\0
\end{tabular}\endgroup%
}}\right.$}%
\begingroup \smaller\smaller\smaller\begin{tabular}{@{}c@{}}%
30\\-11\\-1
\end{tabular}\endgroup%
\kern3pt%
\begingroup \smaller\smaller\smaller\begin{tabular}{@{}c@{}}%
90\\-30\\-8
\end{tabular}\endgroup%
\kern3pt%
\begingroup \smaller\smaller\smaller\begin{tabular}{@{}c@{}}%
90\\-27\\-11
\end{tabular}\endgroup%
\kern3pt%
\begingroup \smaller\smaller\smaller\begin{tabular}{@{}c@{}}%
5\\-1\\-1
\end{tabular}\endgroup%
\kern3pt%
\begingroup \smaller\smaller\smaller\begin{tabular}{@{}c@{}}%
9\\0\\-2
\end{tabular}\endgroup%
\kern3pt%
\begingroup \smaller\smaller\smaller\begin{tabular}{@{}c@{}}%
5\\1\\-1
\end{tabular}\endgroup%
\kern3pt%
\begingroup \smaller\smaller\smaller\begin{tabular}{@{}c@{}}%
9\\3\\-1
\end{tabular}\endgroup%
\kern3pt%
\begingroup \smaller\smaller\smaller\begin{tabular}{@{}c@{}}%
30\\11\\-1
\end{tabular}\endgroup%
{$\left.\llap{\phantom{%
\begingroup \smaller\smaller\smaller\begin{tabular}{@{}c@{}}%
0\\0\\0
\end{tabular}\endgroup%
}}\!\right]$}%
}%
\ifdim\wd\matricesbox>\halfwidth\myboxwidth=\hsize\else\myboxwidth=\halfwidth\fi
\vbox{%
\ifdim\myboxwidth=\hsize
\setbox\onelinebox=\hbox{%
\vbox{\hbox{%
$\Pi_{16,20}$ spans $L_{16.13}$%
}\hbox{%
$36\slashthree6322222\slashthree22222\rtimes D_{2}$%
}%
}%
\hfill\copy\matricesbox
}%
\ifdim\wd\onelinebox>\myboxwidth
\hbox to \myboxwidth{%
$\Pi_{16,20}$ spans $L_{16.13}$%
\hfil
$36\slashthree6322222\slashthree22222\rtimes D_{2}$%
}%
\box\matricesbox
\else
\hbox to \myboxwidth{%
\unhbox\onelinebox
}%
\fi
\else
\hbox to \myboxwidth{%
$\Pi_{16,20}$ spans $L_{16.13}$%
\hfil}%
\hbox to \myboxwidth{%
$36\slashthree6322222\slashthree22222\rtimes D_{2}$%
\hfil}%
\box\matricesbox
\fi
}%
\hfill\discretionary{}{}{}%
\setbox\matricesbox=\hbox{%
{$\left[\!\llap{\phantom{%
\begingroup \smaller\smaller\smaller\begin{tabular}{@{}c@{}}%
\phantom{0}\\\phantom{0}\\\phantom{0}
\end{tabular}\endgroup%
}}\right.$}%
\begingroup \smaller\smaller\smaller\begin{tabular}{@{}c@{}}%
-1\\\phantom{0}\\\phantom{0}
\end{tabular}\endgroup%
\kern3pt%
\begingroup \smaller\smaller\smaller\begin{tabular}{@{}c@{}}%
\phantom{0}\\9\\\phantom{0}
\end{tabular}\endgroup%
\kern3pt%
\begingroup \smaller\smaller\smaller\begin{tabular}{@{}c@{}}%
\phantom{0}\\\phantom{0}\\9
\end{tabular}\endgroup%
{$\left.\llap{\phantom{%
\begingroup \smaller\smaller\smaller\begin{tabular}{@{}c@{}}%
\phantom{0}\\\phantom{0}\\\phantom{0}
\end{tabular}\endgroup%
}}\!\right]$}%
{$\left[\!\llap{\phantom{%
\begingroup \smaller\smaller\smaller\begin{tabular}{@{}c@{}}%
0\\0\\0
\end{tabular}\endgroup%
}}\right.$}%
\begingroup \smaller\smaller\smaller\begin{tabular}{@{}c@{}}%
9\\3\\-1
\end{tabular}\endgroup%
\kern3pt%
\begingroup \smaller\smaller\smaller\begin{tabular}{@{}c@{}}%
8\\2\\-2
\end{tabular}\endgroup%
\kern3pt%
\begingroup \smaller\smaller\smaller\begin{tabular}{@{}c@{}}%
72\\10\\-22
\end{tabular}\endgroup%
\kern3pt%
\begingroup \smaller\smaller\smaller\begin{tabular}{@{}c@{}}%
36\\2\\-12
\end{tabular}\endgroup%
\kern3pt%
\begingroup \smaller\smaller\smaller\begin{tabular}{@{}c@{}}%
9\\-1\\-3
\end{tabular}\endgroup%
\kern3pt%
\begingroup \smaller\smaller\smaller\begin{tabular}{@{}c@{}}%
8\\-2\\-2
\end{tabular}\endgroup%
\kern3pt%
\begingroup \smaller\smaller\smaller\begin{tabular}{@{}c@{}}%
72\\-22\\-10
\end{tabular}\endgroup%
\kern3pt%
\begingroup \smaller\smaller\smaller\begin{tabular}{@{}c@{}}%
36\\-12\\-2
\end{tabular}\endgroup%
{$\left.\llap{\phantom{%
\begingroup \smaller\smaller\smaller\begin{tabular}{@{}c@{}}%
0\\0\\0
\end{tabular}\endgroup%
}}\!\right]$}%
}%
\ifdim\wd\matricesbox>\halfwidth\myboxwidth=\hsize\else\myboxwidth=\halfwidth\fi
\vbox{%
\ifdim\myboxwidth=\hsize
\setbox\onelinebox=\hbox{%
\vbox{\hbox{%
$\Pi_{16,21}$ spans $L_{142.20}$%
}\hbox{%
$\infty422\slashinfty224\infty224\slashinfty422\rtimes D_{2}$%
}%
}%
\hfill\copy\matricesbox
}%
\ifdim\wd\onelinebox>\myboxwidth
\hbox to \myboxwidth{%
$\Pi_{16,21}$ spans $L_{142.20}$%
\hfil
$\infty422\slashinfty224\infty224\slashinfty422\rtimes D_{2}$%
}%
\box\matricesbox
\else
\hbox to \myboxwidth{%
\unhbox\onelinebox
}%
\fi
\else
\hbox to \myboxwidth{%
$\Pi_{16,21}$ spans $L_{142.20}$%
\hfil}%
\hbox to \myboxwidth{%
$\infty422\slashinfty224\infty224\slashinfty422\rtimes D_{2}$%
\hfil}%
\box\matricesbox
\fi
}%
\hfill\discretionary{}{}{}%
\setbox\matricesbox=\hbox{%
{$\left[\!\llap{\phantom{%
\begingroup \smaller\smaller\smaller\begin{tabular}{@{}c@{}}%
\phantom{0}\\\phantom{0}\\\phantom{0}
\end{tabular}\endgroup%
}}\right.$}%
\begingroup \smaller\smaller\smaller\begin{tabular}{@{}c@{}}%
-1\\\phantom{0}\\\phantom{0}
\end{tabular}\endgroup%
\kern3pt%
\begingroup \smaller\smaller\smaller\begin{tabular}{@{}c@{}}%
\phantom{0}\\18\\\phantom{0}
\end{tabular}\endgroup%
\kern3pt%
\begingroup \smaller\smaller\smaller\begin{tabular}{@{}c@{}}%
\phantom{0}\\\phantom{0}\\18
\end{tabular}\endgroup%
{$\left.\llap{\phantom{%
\begingroup \smaller\smaller\smaller\begin{tabular}{@{}c@{}}%
\phantom{0}\\\phantom{0}\\\phantom{0}
\end{tabular}\endgroup%
}}\!\right]$}%
{$\left[\!\llap{\phantom{%
\begingroup \smaller\smaller\smaller\begin{tabular}{@{}c@{}}%
0\\0\\0
\end{tabular}\endgroup%
}}\right.$}%
\begingroup \smaller\smaller\smaller\begin{tabular}{@{}c@{}}%
8\\2\\0
\end{tabular}\endgroup%
\kern3pt%
\begingroup \smaller\smaller\smaller\begin{tabular}{@{}c@{}}%
72\\16\\6
\end{tabular}\endgroup%
\kern3pt%
\begingroup \smaller\smaller\smaller\begin{tabular}{@{}c@{}}%
36\\7\\5
\end{tabular}\endgroup%
\kern3pt%
\begingroup \smaller\smaller\smaller\begin{tabular}{@{}c@{}}%
9\\1\\2
\end{tabular}\endgroup%
\kern3pt%
\begingroup \smaller\smaller\smaller\begin{tabular}{@{}c@{}}%
8\\0\\2
\end{tabular}\endgroup%
\kern3pt%
\begingroup \smaller\smaller\smaller\begin{tabular}{@{}c@{}}%
72\\-6\\16
\end{tabular}\endgroup%
\kern3pt%
\begingroup \smaller\smaller\smaller\begin{tabular}{@{}c@{}}%
36\\-5\\7
\end{tabular}\endgroup%
\kern3pt%
\begingroup \smaller\smaller\smaller\begin{tabular}{@{}c@{}}%
9\\-2\\1
\end{tabular}\endgroup%
\kern3pt%
\begingroup \smaller\smaller\smaller\begin{tabular}{@{}c@{}}%
8\\-2\\0
\end{tabular}\endgroup%
{$\left.\llap{\phantom{%
\begingroup \smaller\smaller\smaller\begin{tabular}{@{}c@{}}%
0\\0\\0
\end{tabular}\endgroup%
}}\!\right]$}%
}%
\ifdim\wd\matricesbox>\halfwidth\myboxwidth=\hsize\else\myboxwidth=\halfwidth\fi
\vbox{%
\ifdim\myboxwidth=\hsize
\setbox\onelinebox=\hbox{%
\vbox{\hbox{%
$\Pi_{16,22}$ spans $L_{142.20}$%
}\hbox{%
$\infty422\infty42|24\infty224\infty2|2\rtimes D_{2}$%
}%
}%
\hfill\copy\matricesbox
}%
\ifdim\wd\onelinebox>\myboxwidth
\hbox to \myboxwidth{%
$\Pi_{16,22}$ spans $L_{142.20}$%
\hfil
$\infty422\infty42|24\infty224\infty2|2\rtimes D_{2}$%
}%
\box\matricesbox
\else
\hbox to \myboxwidth{%
\unhbox\onelinebox
}%
\fi
\else
\hbox to \myboxwidth{%
$\Pi_{16,22}$ spans $L_{142.20}$%
\hfil}%
\hbox to \myboxwidth{%
$\infty422\infty42|24\infty224\infty2|2\rtimes D_{2}$%
\hfil}%
\box\matricesbox
\fi
}%
\hfill\discretionary{}{}{}%
\setbox\matricesbox=\hbox{%
{$\left[\!\llap{\phantom{%
\begingroup \smaller\smaller\smaller\begin{tabular}{@{}c@{}}%
\phantom{0}\\\phantom{0}\\\phantom{0}
\end{tabular}\endgroup%
}}\right.$}%
\begingroup \smaller\smaller\smaller\begin{tabular}{@{}c@{}}%
-1\\\phantom{0}\\\phantom{0}
\end{tabular}\endgroup%
\kern3pt%
\begingroup \smaller\smaller\smaller\begin{tabular}{@{}c@{}}%
\phantom{0}\\9\\\phantom{0}
\end{tabular}\endgroup%
\kern3pt%
\begingroup \smaller\smaller\smaller\begin{tabular}{@{}c@{}}%
\phantom{0}\\\phantom{0}\\9
\end{tabular}\endgroup%
{$\left.\llap{\phantom{%
\begingroup \smaller\smaller\smaller\begin{tabular}{@{}c@{}}%
\phantom{0}\\\phantom{0}\\\phantom{0}
\end{tabular}\endgroup%
}}\!\right]$}%
{$\left[\!\llap{\phantom{%
\begingroup \smaller\smaller\smaller\begin{tabular}{@{}c@{}}%
0\\0\\0
\end{tabular}\endgroup%
}}\right.$}%
\begingroup \smaller\smaller\smaller\begin{tabular}{@{}c@{}}%
36\\12\\-2
\end{tabular}\endgroup%
\kern3pt%
\begingroup \smaller\smaller\smaller\begin{tabular}{@{}c@{}}%
72\\22\\-10
\end{tabular}\endgroup%
\kern3pt%
\begingroup \smaller\smaller\smaller\begin{tabular}{@{}c@{}}%
8\\2\\-2
\end{tabular}\endgroup%
\kern3pt%
\begingroup \smaller\smaller\smaller\begin{tabular}{@{}c@{}}%
72\\10\\-22
\end{tabular}\endgroup%
\kern3pt%
\begingroup \smaller\smaller\smaller\begin{tabular}{@{}c@{}}%
36\\2\\-12
\end{tabular}\endgroup%
\kern3pt%
\begingroup \smaller\smaller\smaller\begin{tabular}{@{}c@{}}%
9\\-1\\-3
\end{tabular}\endgroup%
\kern3pt%
\begingroup \smaller\smaller\smaller\begin{tabular}{@{}c@{}}%
8\\-2\\-2
\end{tabular}\endgroup%
\kern3pt%
\begingroup \smaller\smaller\smaller\begin{tabular}{@{}c@{}}%
9\\-3\\-1
\end{tabular}\endgroup%
{$\left.\llap{\phantom{%
\begingroup \smaller\smaller\smaller\begin{tabular}{@{}c@{}}%
0\\0\\0
\end{tabular}\endgroup%
}}\!\right]$}%
}%
\ifdim\wd\matricesbox>\halfwidth\myboxwidth=\hsize\else\myboxwidth=\halfwidth\fi
\vbox{%
\ifdim\myboxwidth=\hsize
\setbox\onelinebox=\hbox{%
\vbox{\hbox{%
$\Pi_{16,23}$ spans $L_{142.20}$%
}\hbox{%
$\infty4224\slashinfty4224\infty22\slashinfty22\rtimes D_{2}$%
}%
}%
\hfill\copy\matricesbox
}%
\ifdim\wd\onelinebox>\myboxwidth
\hbox to \myboxwidth{%
$\Pi_{16,23}$ spans $L_{142.20}$%
\hfil
$\infty4224\slashinfty4224\infty22\slashinfty22\rtimes D_{2}$%
}%
\box\matricesbox
\else
\hbox to \myboxwidth{%
\unhbox\onelinebox
}%
\fi
\else
\hbox to \myboxwidth{%
$\Pi_{16,23}$ spans $L_{142.20}$%
\hfil}%
\hbox to \myboxwidth{%
$\infty4224\slashinfty4224\infty22\slashinfty22\rtimes D_{2}$%
\hfil}%
\box\matricesbox
\fi
}%
\hfill\discretionary{}{}{}%
\setbox\matricesbox=\hbox{%
{$\left[\!\llap{\phantom{%
\begingroup \smaller\smaller\smaller\begin{tabular}{@{}c@{}}%
\phantom{0}\\\phantom{0}\\\phantom{0}
\end{tabular}\endgroup%
}}\right.$}%
\begingroup \smaller\smaller\smaller\begin{tabular}{@{}c@{}}%
-1\\\phantom{0}\\\phantom{0}
\end{tabular}\endgroup%
\kern3pt%
\begingroup \smaller\smaller\smaller\begin{tabular}{@{}c@{}}%
\phantom{0}\\45/2\\\phantom{0}
\end{tabular}\endgroup%
\kern3pt%
\begingroup \smaller\smaller\smaller\begin{tabular}{@{}c@{}}%
\phantom{0}\\\phantom{0}\\15/2
\end{tabular}\endgroup%
{$\left.\llap{\phantom{%
\begingroup \smaller\smaller\smaller\begin{tabular}{@{}c@{}}%
\phantom{0}\\\phantom{0}\\\phantom{0}
\end{tabular}\endgroup%
}}\!\right]$}%
{$\left[\!\llap{\phantom{%
\begingroup \smaller\smaller\smaller\begin{tabular}{@{}c@{}}%
0\\0\\0
\end{tabular}\endgroup%
}}\right.$}%
\begingroup \smaller\smaller\smaller\begin{tabular}{@{}c@{}}%
9\\2\\0
\end{tabular}\endgroup%
\kern3pt%
\begingroup \smaller\smaller\smaller\begin{tabular}{@{}c@{}}%
5\\1\\-1
\end{tabular}\endgroup%
\kern3pt%
\begingroup \smaller\smaller\smaller\begin{tabular}{@{}c@{}}%
9\\1\\-3
\end{tabular}\endgroup%
\kern3pt%
\begingroup \smaller\smaller\smaller\begin{tabular}{@{}c@{}}%
30\\1\\-11
\end{tabular}\endgroup%
\kern3pt%
\begingroup \smaller\smaller\smaller\begin{tabular}{@{}c@{}}%
30\\-1\\-11
\end{tabular}\endgroup%
\kern3pt%
\begingroup \smaller\smaller\smaller\begin{tabular}{@{}c@{}}%
9\\-1\\-3
\end{tabular}\endgroup%
\kern3pt%
\begingroup \smaller\smaller\smaller\begin{tabular}{@{}c@{}}%
30\\-5\\-7
\end{tabular}\endgroup%
\kern3pt%
\begingroup \smaller\smaller\smaller\begin{tabular}{@{}c@{}}%
30\\-6\\-4
\end{tabular}\endgroup%
\kern3pt%
\begingroup \smaller\smaller\smaller\begin{tabular}{@{}c@{}}%
9\\-2\\0
\end{tabular}\endgroup%
{$\left.\llap{\phantom{%
\begingroup \smaller\smaller\smaller\begin{tabular}{@{}c@{}}%
0\\0\\0
\end{tabular}\endgroup%
}}\!\right]$}%
}%
\ifdim\wd\matricesbox>\halfwidth\myboxwidth=\hsize\else\myboxwidth=\halfwidth\fi
\vbox{%
\ifdim\myboxwidth=\hsize
\setbox\onelinebox=\hbox{%
\vbox{\hbox{%
$\Pi_{16,24}$ spans $L_{16.13}$%
}\hbox{%
$3222|22232232|2322\rtimes D_{2}$%
}%
}%
\hfill\copy\matricesbox
}%
\ifdim\wd\onelinebox>\myboxwidth
\hbox to \myboxwidth{%
$\Pi_{16,24}$ spans $L_{16.13}$%
\hfil
$3222|22232232|2322\rtimes D_{2}$%
}%
\box\matricesbox
\else
\hbox to \myboxwidth{%
\unhbox\onelinebox
}%
\fi
\else
\hbox to \myboxwidth{%
$\Pi_{16,24}$ spans $L_{16.13}$%
\hfil}%
\hbox to \myboxwidth{%
$3222|22232232|2322\rtimes D_{2}$%
\hfil}%
\box\matricesbox
\fi
}%
\hfill\discretionary{}{}{}%
\setbox\matricesbox=\hbox{%
{$\left[\!\llap{\phantom{%
\begingroup \smaller\smaller\smaller\begin{tabular}{@{}c@{}}%
\phantom{0}\\\phantom{0}\\\phantom{0}
\end{tabular}\endgroup%
}}\right.$}%
\begingroup \smaller\smaller\smaller\begin{tabular}{@{}c@{}}%
-1\\\phantom{0}\\\phantom{0}
\end{tabular}\endgroup%
\kern3pt%
\begingroup \smaller\smaller\smaller\begin{tabular}{@{}c@{}}%
\phantom{0}\\15/2\\\phantom{0}
\end{tabular}\endgroup%
\kern3pt%
\begingroup \smaller\smaller\smaller\begin{tabular}{@{}c@{}}%
\phantom{0}\\\phantom{0}\\45/2
\end{tabular}\endgroup%
{$\left.\llap{\phantom{%
\begingroup \smaller\smaller\smaller\begin{tabular}{@{}c@{}}%
\phantom{0}\\\phantom{0}\\\phantom{0}
\end{tabular}\endgroup%
}}\!\right]$}%
{$\left[\!\llap{\phantom{%
\begingroup \smaller\smaller\smaller\begin{tabular}{@{}c@{}}%
0\\0\\0
\end{tabular}\endgroup%
}}\right.$}%
\begingroup \smaller\smaller\smaller\begin{tabular}{@{}c@{}}%
30\\11\\1
\end{tabular}\endgroup%
\kern3pt%
\begingroup \smaller\smaller\smaller\begin{tabular}{@{}c@{}}%
9\\3\\1
\end{tabular}\endgroup%
\kern3pt%
\begingroup \smaller\smaller\smaller\begin{tabular}{@{}c@{}}%
5\\1\\1
\end{tabular}\endgroup%
\kern3pt%
\begingroup \smaller\smaller\smaller\begin{tabular}{@{}c@{}}%
9\\0\\2
\end{tabular}\endgroup%
\kern3pt%
\begingroup \smaller\smaller\smaller\begin{tabular}{@{}c@{}}%
30\\-4\\6
\end{tabular}\endgroup%
\kern3pt%
\begingroup \smaller\smaller\smaller\begin{tabular}{@{}c@{}}%
30\\-7\\5
\end{tabular}\endgroup%
\kern3pt%
\begingroup \smaller\smaller\smaller\begin{tabular}{@{}c@{}}%
9\\-3\\1
\end{tabular}\endgroup%
\kern3pt%
\begingroup \smaller\smaller\smaller\begin{tabular}{@{}c@{}}%
30\\-11\\1
\end{tabular}\endgroup%
{$\left.\llap{\phantom{%
\begingroup \smaller\smaller\smaller\begin{tabular}{@{}c@{}}%
0\\0\\0
\end{tabular}\endgroup%
}}\!\right]$}%
}%
\ifdim\wd\matricesbox>\halfwidth\myboxwidth=\hsize\else\myboxwidth=\halfwidth\fi
\vbox{%
\ifdim\myboxwidth=\hsize
\setbox\onelinebox=\hbox{%
\vbox{\hbox{%
$\Pi_{16,25}$ spans $L_{16.13}$%
}\hbox{%
$32222\slashthree2222322\slashthree22\rtimes D_{2}$%
}%
}%
\hfill\copy\matricesbox
}%
\ifdim\wd\onelinebox>\myboxwidth
\hbox to \myboxwidth{%
$\Pi_{16,25}$ spans $L_{16.13}$%
\hfil
$32222\slashthree2222322\slashthree22\rtimes D_{2}$%
}%
\box\matricesbox
\else
\hbox to \myboxwidth{%
\unhbox\onelinebox
}%
\fi
\else
\hbox to \myboxwidth{%
$\Pi_{16,25}$ spans $L_{16.13}$%
\hfil}%
\hbox to \myboxwidth{%
$32222\slashthree2222322\slashthree22\rtimes D_{2}$%
\hfil}%
\box\matricesbox
\fi
}%
\hfill\discretionary{}{}{}%
\setbox\matricesbox=\hbox{%
{$\left[\!\llap{\phantom{%
\begingroup \smaller\smaller\smaller\begin{tabular}{@{}c@{}}%
\phantom{0}\\\phantom{0}\\\phantom{0}
\end{tabular}\endgroup%
}}\right.$}%
\begingroup \smaller\smaller\smaller\begin{tabular}{@{}c@{}}%
-1\\\phantom{0}\\\phantom{0}
\end{tabular}\endgroup%
\kern3pt%
\begingroup \smaller\smaller\smaller\begin{tabular}{@{}c@{}}%
\phantom{0}\\15/2\\\phantom{0}
\end{tabular}\endgroup%
\kern3pt%
\begingroup \smaller\smaller\smaller\begin{tabular}{@{}c@{}}%
\phantom{0}\\\phantom{0}\\45/2
\end{tabular}\endgroup%
{$\left.\llap{\phantom{%
\begingroup \smaller\smaller\smaller\begin{tabular}{@{}c@{}}%
\phantom{0}\\\phantom{0}\\\phantom{0}
\end{tabular}\endgroup%
}}\!\right]$}%
{$\left[\!\llap{\phantom{%
\begingroup \smaller\smaller\smaller\begin{tabular}{@{}c@{}}%
0\\0\\0
\end{tabular}\endgroup%
}}\right.$}%
\begingroup \smaller\smaller\smaller\begin{tabular}{@{}c@{}}%
5\\2\\0
\end{tabular}\endgroup%
\kern3pt%
\begingroup \smaller\smaller\smaller\begin{tabular}{@{}c@{}}%
9\\3\\1
\end{tabular}\endgroup%
\kern3pt%
\begingroup \smaller\smaller\smaller\begin{tabular}{@{}c@{}}%
30\\7\\5
\end{tabular}\endgroup%
\kern3pt%
\begingroup \smaller\smaller\smaller\begin{tabular}{@{}c@{}}%
30\\4\\6
\end{tabular}\endgroup%
\kern3pt%
\begingroup \smaller\smaller\smaller\begin{tabular}{@{}c@{}}%
9\\0\\2
\end{tabular}\endgroup%
\kern3pt%
\begingroup \smaller\smaller\smaller\begin{tabular}{@{}c@{}}%
5\\-1\\1
\end{tabular}\endgroup%
\kern3pt%
\begingroup \smaller\smaller\smaller\begin{tabular}{@{}c@{}}%
90\\-27\\11
\end{tabular}\endgroup%
\kern3pt%
\begingroup \smaller\smaller\smaller\begin{tabular}{@{}c@{}}%
90\\-30\\8
\end{tabular}\endgroup%
\kern3pt%
\begingroup \smaller\smaller\smaller\begin{tabular}{@{}c@{}}%
5\\-2\\0
\end{tabular}\endgroup%
{$\left.\llap{\phantom{%
\begingroup \smaller\smaller\smaller\begin{tabular}{@{}c@{}}%
0\\0\\0
\end{tabular}\endgroup%
}}\!\right]$}%
}%
\ifdim\wd\matricesbox>\halfwidth\myboxwidth=\hsize\else\myboxwidth=\halfwidth\fi
\vbox{%
\ifdim\myboxwidth=\hsize
\setbox\onelinebox=\hbox{%
\vbox{\hbox{%
$\Pi_{16,26}$ spans $L_{16.13}$%
}\hbox{%
$322|22322232|23222\rtimes D_{2}$%
}%
}%
\hfill\copy\matricesbox
}%
\ifdim\wd\onelinebox>\myboxwidth
\hbox to \myboxwidth{%
$\Pi_{16,26}$ spans $L_{16.13}$%
\hfil
$322|22322232|23222\rtimes D_{2}$%
}%
\box\matricesbox
\else
\hbox to \myboxwidth{%
\unhbox\onelinebox
}%
\fi
\else
\hbox to \myboxwidth{%
$\Pi_{16,26}$ spans $L_{16.13}$%
\hfil}%
\hbox to \myboxwidth{%
$322|22322232|23222\rtimes D_{2}$%
\hfil}%
\box\matricesbox
\fi
}%
\hfill\discretionary{}{}{}%
\setbox\matricesbox=\hbox{%
{$\left[\!\llap{\phantom{%
\begingroup \smaller\smaller\smaller\begin{tabular}{@{}c@{}}%
\phantom{0}\\\phantom{0}\\\phantom{0}
\end{tabular}\endgroup%
}}\right.$}%
\begingroup \smaller\smaller\smaller\begin{tabular}{@{}c@{}}%
-1\\\phantom{0}\\\phantom{0}
\end{tabular}\endgroup%
\kern3pt%
\begingroup \smaller\smaller\smaller\begin{tabular}{@{}c@{}}%
\phantom{0}\\45/2\\\phantom{0}
\end{tabular}\endgroup%
\kern3pt%
\begingroup \smaller\smaller\smaller\begin{tabular}{@{}c@{}}%
\phantom{0}\\\phantom{0}\\15/2
\end{tabular}\endgroup%
{$\left.\llap{\phantom{%
\begingroup \smaller\smaller\smaller\begin{tabular}{@{}c@{}}%
\phantom{0}\\\phantom{0}\\\phantom{0}
\end{tabular}\endgroup%
}}\!\right]$}%
{$\left[\!\llap{\phantom{%
\begingroup \smaller\smaller\smaller\begin{tabular}{@{}c@{}}%
0\\0\\0
\end{tabular}\endgroup%
}}\right.$}%
\begingroup \smaller\smaller\smaller\begin{tabular}{@{}c@{}}%
9\\2\\0
\end{tabular}\endgroup%
\kern3pt%
\begingroup \smaller\smaller\smaller\begin{tabular}{@{}c@{}}%
5\\1\\1
\end{tabular}\endgroup%
\kern3pt%
\begingroup \smaller\smaller\smaller\begin{tabular}{@{}c@{}}%
90\\11\\27
\end{tabular}\endgroup%
\kern3pt%
\begingroup \smaller\smaller\smaller\begin{tabular}{@{}c@{}}%
90\\8\\30
\end{tabular}\endgroup%
\kern3pt%
\begingroup \smaller\smaller\smaller\begin{tabular}{@{}c@{}}%
5\\0\\2
\end{tabular}\endgroup%
\kern3pt%
\begingroup \smaller\smaller\smaller\begin{tabular}{@{}c@{}}%
9\\-1\\3
\end{tabular}\endgroup%
\kern3pt%
\begingroup \smaller\smaller\smaller\begin{tabular}{@{}c@{}}%
30\\-5\\7
\end{tabular}\endgroup%
\kern3pt%
\begingroup \smaller\smaller\smaller\begin{tabular}{@{}c@{}}%
30\\-6\\4
\end{tabular}\endgroup%
\kern3pt%
\begingroup \smaller\smaller\smaller\begin{tabular}{@{}c@{}}%
9\\-2\\0
\end{tabular}\endgroup%
{$\left.\llap{\phantom{%
\begingroup \smaller\smaller\smaller\begin{tabular}{@{}c@{}}%
0\\0\\0
\end{tabular}\endgroup%
}}\!\right]$}%
}%
\ifdim\wd\matricesbox>\halfwidth\myboxwidth=\hsize\else\myboxwidth=\halfwidth\fi
\vbox{%
\ifdim\myboxwidth=\hsize
\setbox\onelinebox=\hbox{%
\vbox{\hbox{%
$\Pi_{16,27}$ spans $L_{16.13}$%
}\hbox{%
$3222322|22322232|2\rtimes D_{2}$%
}%
}%
\hfill\copy\matricesbox
}%
\ifdim\wd\onelinebox>\myboxwidth
\hbox to \myboxwidth{%
$\Pi_{16,27}$ spans $L_{16.13}$%
\hfil
$3222322|22322232|2\rtimes D_{2}$%
}%
\box\matricesbox
\else
\hbox to \myboxwidth{%
\unhbox\onelinebox
}%
\fi
\else
\hbox to \myboxwidth{%
$\Pi_{16,27}$ spans $L_{16.13}$%
\hfil}%
\hbox to \myboxwidth{%
$3222322|22322232|2\rtimes D_{2}$%
\hfil}%
\box\matricesbox
\fi
}%
\hfill\discretionary{}{}{}%
\setbox\matricesbox=\hbox{%
{$\left[\!\llap{\phantom{%
\begingroup \smaller\smaller\smaller\begin{tabular}{@{}c@{}}%
\phantom{0}\\\phantom{0}\\\phantom{0}
\end{tabular}\endgroup%
}}\right.$}%
\begingroup \smaller\smaller\smaller\begin{tabular}{@{}c@{}}%
-1\\\phantom{0}\\\phantom{0}
\end{tabular}\endgroup%
\kern3pt%
\begingroup \smaller\smaller\smaller\begin{tabular}{@{}c@{}}%
\phantom{0}\\21/2\\\phantom{0}
\end{tabular}\endgroup%
\kern3pt%
\begingroup \smaller\smaller\smaller\begin{tabular}{@{}c@{}}%
\phantom{0}\\\phantom{0}\\7/2
\end{tabular}\endgroup%
{$\left.\llap{\phantom{%
\begingroup \smaller\smaller\smaller\begin{tabular}{@{}c@{}}%
\phantom{0}\\\phantom{0}\\\phantom{0}
\end{tabular}\endgroup%
}}\!\right]$}%
{$\left[\!\llap{\phantom{%
\begingroup \smaller\smaller\smaller\begin{tabular}{@{}c@{}}%
0\\0\\0
\end{tabular}\endgroup%
}}\right.$}%
\begingroup \smaller\smaller\smaller\begin{tabular}{@{}c@{}}%
42\\13\\-3
\end{tabular}\endgroup%
\kern3pt%
\begingroup \smaller\smaller\smaller\begin{tabular}{@{}c@{}}%
7\\2\\-2
\end{tabular}\endgroup%
\kern3pt%
\begingroup \smaller\smaller\smaller\begin{tabular}{@{}c@{}}%
6\\1\\-3
\end{tabular}\endgroup%
\kern3pt%
\begingroup \smaller\smaller\smaller\begin{tabular}{@{}c@{}}%
7\\0\\-4
\end{tabular}\endgroup%
\kern3pt%
\begingroup \smaller\smaller\smaller\begin{tabular}{@{}c@{}}%
42\\-5\\-21
\end{tabular}\endgroup%
\kern3pt%
\begingroup \smaller\smaller\smaller\begin{tabular}{@{}c@{}}%
42\\-8\\-18
\end{tabular}\endgroup%
\kern3pt%
\begingroup \smaller\smaller\smaller\begin{tabular}{@{}c@{}}%
7\\-2\\-2
\end{tabular}\endgroup%
\kern3pt%
\begingroup \smaller\smaller\smaller\begin{tabular}{@{}c@{}}%
42\\-13\\-3
\end{tabular}\endgroup%
{$\left.\llap{\phantom{%
\begingroup \smaller\smaller\smaller\begin{tabular}{@{}c@{}}%
0\\0\\0
\end{tabular}\endgroup%
}}\!\right]$}%
}%
\ifdim\wd\matricesbox>\halfwidth\myboxwidth=\hsize\else\myboxwidth=\halfwidth\fi
\vbox{%
\ifdim\myboxwidth=\hsize
\setbox\onelinebox=\hbox{%
\vbox{\hbox{%
$\Pi_{16,28}$ spans $L_{22.4}$%
}\hbox{%
$32222\slashthree2222322\slashthree22\rtimes D_{2}$%
}%
}%
\hfill\copy\matricesbox
}%
\ifdim\wd\onelinebox>\myboxwidth
\hbox to \myboxwidth{%
$\Pi_{16,28}$ spans $L_{22.4}$%
\hfil
$32222\slashthree2222322\slashthree22\rtimes D_{2}$%
}%
\box\matricesbox
\else
\hbox to \myboxwidth{%
\unhbox\onelinebox
}%
\fi
\else
\hbox to \myboxwidth{%
$\Pi_{16,28}$ spans $L_{22.4}$%
\hfil}%
\hbox to \myboxwidth{%
$32222\slashthree2222322\slashthree22\rtimes D_{2}$%
\hfil}%
\box\matricesbox
\fi
}%
\hfill\discretionary{}{}{}%
\setbox\matricesbox=\hbox{%
{$\left[\!\llap{\phantom{%
\begingroup \smaller\smaller\smaller\begin{tabular}{@{}c@{}}%
\phantom{0}\\\phantom{0}\\\phantom{0}
\end{tabular}\endgroup%
}}\right.$}%
\begingroup \smaller\smaller\smaller\begin{tabular}{@{}c@{}}%
-1\\\phantom{0}\\\phantom{0}
\end{tabular}\endgroup%
\kern3pt%
\begingroup \smaller\smaller\smaller\begin{tabular}{@{}c@{}}%
\phantom{0}\\45/2\\\phantom{0}
\end{tabular}\endgroup%
\kern3pt%
\begingroup \smaller\smaller\smaller\begin{tabular}{@{}c@{}}%
\phantom{0}\\\phantom{0}\\15/2
\end{tabular}\endgroup%
{$\left.\llap{\phantom{%
\begingroup \smaller\smaller\smaller\begin{tabular}{@{}c@{}}%
\phantom{0}\\\phantom{0}\\\phantom{0}
\end{tabular}\endgroup%
}}\!\right]$}%
{$\left[\!\llap{\phantom{%
\begingroup \smaller\smaller\smaller\begin{tabular}{@{}c@{}}%
0\\0\\0
\end{tabular}\endgroup%
}}\right.$}%
\begingroup \smaller\smaller\smaller\begin{tabular}{@{}c@{}}%
90\\-19\\3
\end{tabular}\endgroup%
\kern3pt%
\begingroup \smaller\smaller\smaller\begin{tabular}{@{}c@{}}%
5\\-1\\1
\end{tabular}\endgroup%
\kern3pt%
\begingroup \smaller\smaller\smaller\begin{tabular}{@{}c@{}}%
9\\-1\\3
\end{tabular}\endgroup%
\kern3pt%
\begingroup \smaller\smaller\smaller\begin{tabular}{@{}c@{}}%
5\\0\\2
\end{tabular}\endgroup%
\kern3pt%
\begingroup \smaller\smaller\smaller\begin{tabular}{@{}c@{}}%
90\\8\\30
\end{tabular}\endgroup%
\kern3pt%
\begingroup \smaller\smaller\smaller\begin{tabular}{@{}c@{}}%
90\\11\\27
\end{tabular}\endgroup%
\kern3pt%
\begingroup \smaller\smaller\smaller\begin{tabular}{@{}c@{}}%
5\\1\\1
\end{tabular}\endgroup%
\kern3pt%
\begingroup \smaller\smaller\smaller\begin{tabular}{@{}c@{}}%
90\\19\\3
\end{tabular}\endgroup%
{$\left.\llap{\phantom{%
\begingroup \smaller\smaller\smaller\begin{tabular}{@{}c@{}}%
0\\0\\0
\end{tabular}\endgroup%
}}\!\right]$}%
}%
\ifdim\wd\matricesbox>\halfwidth\myboxwidth=\hsize\else\myboxwidth=\halfwidth\fi
\vbox{%
\ifdim\myboxwidth=\hsize
\setbox\onelinebox=\hbox{%
\vbox{\hbox{%
$\Pi_{16,29}$ spans $L_{16.13}$%
}\hbox{%
$32222\slashthree2222322\slashthree22\rtimes D_{2}$%
}%
}%
\hfill\copy\matricesbox
}%
\ifdim\wd\onelinebox>\myboxwidth
\hbox to \myboxwidth{%
$\Pi_{16,29}$ spans $L_{16.13}$%
\hfil
$32222\slashthree2222322\slashthree22\rtimes D_{2}$%
}%
\box\matricesbox
\else
\hbox to \myboxwidth{%
\unhbox\onelinebox
}%
\fi
\else
\hbox to \myboxwidth{%
$\Pi_{16,29}$ spans $L_{16.13}$%
\hfil}%
\hbox to \myboxwidth{%
$32222\slashthree2222322\slashthree22\rtimes D_{2}$%
\hfil}%
\box\matricesbox
\fi
}%
\hfill\discretionary{}{}{}%
\setbox\matricesbox=\hbox{%
{$\left[\!\llap{\phantom{%
\begingroup \smaller\smaller\smaller\begin{tabular}{@{}c@{}}%
\phantom{0}\\\phantom{0}\\\phantom{0}
\end{tabular}\endgroup%
}}\right.$}%
\begingroup \smaller\smaller\smaller\begin{tabular}{@{}c@{}}%
-1\\\phantom{0}\\\phantom{0}
\end{tabular}\endgroup%
\kern3pt%
\begingroup \smaller\smaller\smaller\begin{tabular}{@{}c@{}}%
\phantom{0}\\18\\\phantom{0}
\end{tabular}\endgroup%
\kern3pt%
\begingroup \smaller\smaller\smaller\begin{tabular}{@{}c@{}}%
\phantom{0}\\\phantom{0}\\18
\end{tabular}\endgroup%
{$\left.\llap{\phantom{%
\begingroup \smaller\smaller\smaller\begin{tabular}{@{}c@{}}%
\phantom{0}\\\phantom{0}\\\phantom{0}
\end{tabular}\endgroup%
}}\!\right]$}%
{$\left[\!\llap{\phantom{%
\begingroup \smaller\smaller\smaller\begin{tabular}{@{}c@{}}%
0\\0\\0
\end{tabular}\endgroup%
}}\right.$}%
\begingroup \smaller\smaller\smaller\begin{tabular}{@{}c@{}}%
8\\2\\0
\end{tabular}\endgroup%
\kern3pt%
\begingroup \smaller\smaller\smaller\begin{tabular}{@{}c@{}}%
9\\2\\1
\end{tabular}\endgroup%
\kern3pt%
\begingroup \smaller\smaller\smaller\begin{tabular}{@{}c@{}}%
9\\1\\2
\end{tabular}\endgroup%
\kern3pt%
\begingroup \smaller\smaller\smaller\begin{tabular}{@{}c@{}}%
8\\0\\2
\end{tabular}\endgroup%
\kern3pt%
\begingroup \smaller\smaller\smaller\begin{tabular}{@{}c@{}}%
72\\-6\\16
\end{tabular}\endgroup%
\kern3pt%
\begingroup \smaller\smaller\smaller\begin{tabular}{@{}c@{}}%
36\\-5\\7
\end{tabular}\endgroup%
\kern3pt%
\begingroup \smaller\smaller\smaller\begin{tabular}{@{}c@{}}%
36\\-7\\5
\end{tabular}\endgroup%
\kern3pt%
\begingroup \smaller\smaller\smaller\begin{tabular}{@{}c@{}}%
72\\-16\\6
\end{tabular}\endgroup%
\kern3pt%
\begingroup \smaller\smaller\smaller\begin{tabular}{@{}c@{}}%
8\\-2\\0
\end{tabular}\endgroup%
{$\left.\llap{\phantom{%
\begingroup \smaller\smaller\smaller\begin{tabular}{@{}c@{}}%
0\\0\\0
\end{tabular}\endgroup%
}}\!\right]$}%
}%
\ifdim\wd\matricesbox>\halfwidth\myboxwidth=\hsize\else\myboxwidth=\halfwidth\fi
\vbox{%
\ifdim\myboxwidth=\hsize
\setbox\onelinebox=\hbox{%
\vbox{\hbox{%
$\Pi_{16,30}$ spans $L_{142.20}$%
}\hbox{%
$422\infty2|2\infty224\infty42|24\infty\rtimes D_{2}$%
}%
}%
\hfill\copy\matricesbox
}%
\ifdim\wd\onelinebox>\myboxwidth
\hbox to \myboxwidth{%
$\Pi_{16,30}$ spans $L_{142.20}$%
\hfil
$422\infty2|2\infty224\infty42|24\infty\rtimes D_{2}$%
}%
\box\matricesbox
\else
\hbox to \myboxwidth{%
\unhbox\onelinebox
}%
\fi
\else
\hbox to \myboxwidth{%
$\Pi_{16,30}$ spans $L_{142.20}$%
\hfil}%
\hbox to \myboxwidth{%
$422\infty2|2\infty224\infty42|24\infty\rtimes D_{2}$%
\hfil}%
\box\matricesbox
\fi
}%
\hfill\discretionary{}{}{}%
\setbox\matricesbox=\hbox{%
{$\left[\!\llap{\phantom{%
\begingroup \smaller\smaller\smaller\begin{tabular}{@{}c@{}}%
\phantom{0}\\\phantom{0}\\\phantom{0}
\end{tabular}\endgroup%
}}\right.$}%
\begingroup \smaller\smaller\smaller\begin{tabular}{@{}c@{}}%
-1\\\phantom{0}\\\phantom{0}
\end{tabular}\endgroup%
\kern3pt%
\begingroup \smaller\smaller\smaller\begin{tabular}{@{}c@{}}%
\phantom{0}\\30\\-15
\end{tabular}\endgroup%
\kern3pt%
\begingroup \smaller\smaller\smaller\begin{tabular}{@{}c@{}}%
\phantom{0}\\-15\\30
\end{tabular}\endgroup%
{$\left.\llap{\phantom{%
\begingroup \smaller\smaller\smaller\begin{tabular}{@{}c@{}}%
\phantom{0}\\\phantom{0}\\\phantom{0}
\end{tabular}\endgroup%
}}\!\right]$}%
{$\left[\!\llap{\phantom{%
\begingroup \smaller\smaller\smaller\begin{tabular}{@{}c@{}}%
0\\0\\0
\end{tabular}\endgroup%
}}\right.$}%
\begingroup \smaller\smaller\smaller\begin{tabular}{@{}c@{}}%
90\\-8\\-19
\end{tabular}\endgroup%
\kern3pt%
\begingroup \smaller\smaller\smaller\begin{tabular}{@{}c@{}}%
90\\-11\\-19
\end{tabular}\endgroup%
\kern3pt%
\begingroup \smaller\smaller\smaller\begin{tabular}{@{}c@{}}%
30\\-5\\-6
\end{tabular}\endgroup%
\kern3pt%
\begingroup \smaller\smaller\smaller\begin{tabular}{@{}c@{}}%
30\\-6\\-5
\end{tabular}\endgroup%
\kern3pt%
\begingroup \smaller\smaller\smaller\begin{tabular}{@{}c@{}}%
9\\-2\\-1
\end{tabular}\endgroup%
\kern3pt%
\begingroup \smaller\smaller\smaller\begin{tabular}{@{}c@{}}%
5\\-1\\0
\end{tabular}\endgroup%
\kern3pt%
\begingroup \smaller\smaller\smaller\begin{tabular}{@{}c@{}}%
9\\-1\\1
\end{tabular}\endgroup%
\kern3pt%
\begingroup \smaller\smaller\smaller\begin{tabular}{@{}c@{}}%
5\\0\\1
\end{tabular}\endgroup%
{$\left.\llap{\phantom{%
\begingroup \smaller\smaller\smaller\begin{tabular}{@{}c@{}}%
0\\0\\0
\end{tabular}\endgroup%
}}\!\right]$}%
}%
\ifdim\wd\matricesbox>\halfwidth\myboxwidth=\hsize\else\myboxwidth=\halfwidth\fi
\vbox{%
\ifdim\myboxwidth=\hsize
\setbox\onelinebox=\hbox{%
\vbox{\hbox{%
$\Pi_{16,31}$ spans $L_{16.13}$%
}\hbox{%
$3632222236322222\rtimes C_{2}$%
}%
}%
\hfill\copy\matricesbox
}%
\ifdim\wd\onelinebox>\myboxwidth
\hbox to \myboxwidth{%
$\Pi_{16,31}$ spans $L_{16.13}$%
\hfil
$3632222236322222\rtimes C_{2}$%
}%
\box\matricesbox
\else
\hbox to \myboxwidth{%
\unhbox\onelinebox
}%
\fi
\else
\hbox to \myboxwidth{%
$\Pi_{16,31}$ spans $L_{16.13}$%
\hfil}%
\hbox to \myboxwidth{%
$3632222236322222\rtimes C_{2}$%
\hfil}%
\box\matricesbox
\fi
}%
\hfill\discretionary{}{}{}%
\setbox\matricesbox=\hbox{%
{$\left[\!\llap{\phantom{%
\begingroup \smaller\smaller\smaller
\endgroup%
}}\!\right]$}%
}%
\ifdim\wd\matricesbox>\halfwidth\myboxwidth=\hsize\else\myboxwidth=\halfwidth\fi
\vbox{%
\ifdim\myboxwidth=\hsize
\setbox\onelinebox=\hbox{%
\vbox{\hbox{%
$\Pi_{16,32}$ spans $L_{16.13}$%
}\hbox{%
$2222222236363632$%
}%
}%
\hfill\copy\matricesbox
}%
\ifdim\wd\onelinebox>\myboxwidth
\hbox to \myboxwidth{%
$\Pi_{16,32}$ spans $L_{16.13}$%
\hfil
$2222222236363632$%
}%
\box\matricesbox
\else
\hbox to \myboxwidth{%
\unhbox\onelinebox
}%
\fi
\else
\hbox to \myboxwidth{%
$\Pi_{16,32}$ spans $L_{16.13}$%
\hfil}%
\hbox to \myboxwidth{%
$2222222236363632$%
\hfil}%
\box\matricesbox
\fi
}%
\hfill\discretionary{}{}{}%
\setbox\matricesbox=\hbox{%
{$\left[\!\llap{\phantom{%
\begingroup \smaller\smaller\smaller
\endgroup%
}}\!\right]$}%
}%
\ifdim\wd\matricesbox>\halfwidth\myboxwidth=\hsize\else\myboxwidth=\halfwidth\fi
\vbox{%
\ifdim\myboxwidth=\hsize
\setbox\onelinebox=\hbox{%
\vbox{\hbox{%
$\Pi_{16,33}$ spans $L_{16.13}$%
}\hbox{%
$2222222322363632$%
}%
}%
\hfill\copy\matricesbox
}%
\ifdim\wd\onelinebox>\myboxwidth
\hbox to \myboxwidth{%
$\Pi_{16,33}$ spans $L_{16.13}$%
\hfil
$2222222322363632$%
}%
\box\matricesbox
\else
\hbox to \myboxwidth{%
\unhbox\onelinebox
}%
\fi
\else
\hbox to \myboxwidth{%
$\Pi_{16,33}$ spans $L_{16.13}$%
\hfil}%
\hbox to \myboxwidth{%
$2222222322363632$%
\hfil}%
\box\matricesbox
\fi
}%
\hfill\discretionary{}{}{}%
\setbox\matricesbox=\hbox{%
{$\left[\!\llap{\phantom{%
\begingroup \smaller\smaller\smaller
\endgroup%
}}\!\right]$}%
}%
\ifdim\wd\matricesbox>\halfwidth\myboxwidth=\hsize\else\myboxwidth=\halfwidth\fi
\vbox{%
\ifdim\myboxwidth=\hsize
\setbox\onelinebox=\hbox{%
\vbox{\hbox{%
$\Pi_{16,34}$ spans $L_{16.13}$%
}\hbox{%
$2222223222363632$%
}%
}%
\hfill\copy\matricesbox
}%
\ifdim\wd\onelinebox>\myboxwidth
\hbox to \myboxwidth{%
$\Pi_{16,34}$ spans $L_{16.13}$%
\hfil
$2222223222363632$%
}%
\box\matricesbox
\else
\hbox to \myboxwidth{%
\unhbox\onelinebox
}%
\fi
\else
\hbox to \myboxwidth{%
$\Pi_{16,34}$ spans $L_{16.13}$%
\hfil}%
\hbox to \myboxwidth{%
$2222223222363632$%
\hfil}%
\box\matricesbox
\fi
}%
\hfill\discretionary{}{}{}%
\setbox\matricesbox=\hbox{%
{$\left[\!\llap{\phantom{%
\begingroup \smaller\smaller\smaller
\endgroup%
}}\!\right]$}%
}%
\ifdim\wd\matricesbox>\halfwidth\myboxwidth=\hsize\else\myboxwidth=\halfwidth\fi
\vbox{%
\ifdim\myboxwidth=\hsize
\setbox\onelinebox=\hbox{%
\vbox{\hbox{%
$\Pi_{16,35}$ spans $L_{16.13}$%
}\hbox{%
$2222223223223632$%
}%
}%
\hfill\copy\matricesbox
}%
\ifdim\wd\onelinebox>\myboxwidth
\hbox to \myboxwidth{%
$\Pi_{16,35}$ spans $L_{16.13}$%
\hfil
$2222223223223632$%
}%
\box\matricesbox
\else
\hbox to \myboxwidth{%
\unhbox\onelinebox
}%
\fi
\else
\hbox to \myboxwidth{%
$\Pi_{16,35}$ spans $L_{16.13}$%
\hfil}%
\hbox to \myboxwidth{%
$2222223223223632$%
\hfil}%
\box\matricesbox
\fi
}%
\hfill\discretionary{}{}{}%
\setbox\matricesbox=\hbox{%
{$\left[\!\llap{\phantom{%
\begingroup \smaller\smaller\smaller
\endgroup%
}}\!\right]$}%
}%
\ifdim\wd\matricesbox>\halfwidth\myboxwidth=\hsize\else\myboxwidth=\halfwidth\fi
\vbox{%
\ifdim\myboxwidth=\hsize
\setbox\onelinebox=\hbox{%
\vbox{\hbox{%
$\Pi_{16,36}$ spans $L_{16.13}$%
}\hbox{%
$2222223223632232$%
}%
}%
\hfill\copy\matricesbox
}%
\ifdim\wd\onelinebox>\myboxwidth
\hbox to \myboxwidth{%
$\Pi_{16,36}$ spans $L_{16.13}$%
\hfil
$2222223223632232$%
}%
\box\matricesbox
\else
\hbox to \myboxwidth{%
\unhbox\onelinebox
}%
\fi
\else
\hbox to \myboxwidth{%
$\Pi_{16,36}$ spans $L_{16.13}$%
\hfil}%
\hbox to \myboxwidth{%
$2222223223632232$%
\hfil}%
\box\matricesbox
\fi
}%
\hfill\discretionary{}{}{}%
\setbox\matricesbox=\hbox{%
{$\left[\!\llap{\phantom{%
\begingroup \smaller\smaller\smaller
\endgroup%
}}\!\right]$}%
}%
\ifdim\wd\matricesbox>\halfwidth\myboxwidth=\hsize\else\myboxwidth=\halfwidth\fi
\vbox{%
\ifdim\myboxwidth=\hsize
\setbox\onelinebox=\hbox{%
\vbox{\hbox{%
$\Pi_{16,37}$ spans $L_{16.13}$%
}\hbox{%
$2222223223636322$%
}%
}%
\hfill\copy\matricesbox
}%
\ifdim\wd\onelinebox>\myboxwidth
\hbox to \myboxwidth{%
$\Pi_{16,37}$ spans $L_{16.13}$%
\hfil
$2222223223636322$%
}%
\box\matricesbox
\else
\hbox to \myboxwidth{%
\unhbox\onelinebox
}%
\fi
\else
\hbox to \myboxwidth{%
$\Pi_{16,37}$ spans $L_{16.13}$%
\hfil}%
\hbox to \myboxwidth{%
$2222223223636322$%
\hfil}%
\box\matricesbox
\fi
}%
\hfill\discretionary{}{}{}%
\setbox\matricesbox=\hbox{%
{$\left[\!\llap{\phantom{%
\begingroup \smaller\smaller\smaller
\endgroup%
}}\!\right]$}%
}%
\ifdim\wd\matricesbox>\halfwidth\myboxwidth=\hsize\else\myboxwidth=\halfwidth\fi
\vbox{%
\ifdim\myboxwidth=\hsize
\setbox\onelinebox=\hbox{%
\vbox{\hbox{%
$\Pi_{16,38}$ spans $L_{16.13}$%
}\hbox{%
$2222223632223632$%
}%
}%
\hfill\copy\matricesbox
}%
\ifdim\wd\onelinebox>\myboxwidth
\hbox to \myboxwidth{%
$\Pi_{16,38}$ spans $L_{16.13}$%
\hfil
$2222223632223632$%
}%
\box\matricesbox
\else
\hbox to \myboxwidth{%
\unhbox\onelinebox
}%
\fi
\else
\hbox to \myboxwidth{%
$\Pi_{16,38}$ spans $L_{16.13}$%
\hfil}%
\hbox to \myboxwidth{%
$2222223632223632$%
\hfil}%
\box\matricesbox
\fi
}%
\hfill\discretionary{}{}{}%
\setbox\matricesbox=\hbox{%
{$\left[\!\llap{\phantom{%
\begingroup \smaller\smaller\smaller
\endgroup%
}}\!\right]$}%
}%
\ifdim\wd\matricesbox>\halfwidth\myboxwidth=\hsize\else\myboxwidth=\halfwidth\fi
\vbox{%
\ifdim\myboxwidth=\hsize
\setbox\onelinebox=\hbox{%
\vbox{\hbox{%
$\Pi_{16,39}$ spans $L_{16.13}$%
}\hbox{%
$2222223632232232$%
}%
}%
\hfill\copy\matricesbox
}%
\ifdim\wd\onelinebox>\myboxwidth
\hbox to \myboxwidth{%
$\Pi_{16,39}$ spans $L_{16.13}$%
\hfil
$2222223632232232$%
}%
\box\matricesbox
\else
\hbox to \myboxwidth{%
\unhbox\onelinebox
}%
\fi
\else
\hbox to \myboxwidth{%
$\Pi_{16,39}$ spans $L_{16.13}$%
\hfil}%
\hbox to \myboxwidth{%
$2222223632232232$%
\hfil}%
\box\matricesbox
\fi
}%
\hfill\discretionary{}{}{}%
\setbox\matricesbox=\hbox{%
{$\left[\!\llap{\phantom{%
\begingroup \smaller\smaller\smaller
\endgroup%
}}\!\right]$}%
}%
\ifdim\wd\matricesbox>\halfwidth\myboxwidth=\hsize\else\myboxwidth=\halfwidth\fi
\vbox{%
\ifdim\myboxwidth=\hsize
\setbox\onelinebox=\hbox{%
\vbox{\hbox{%
$\Pi_{16,40}$ spans $L_{16.13}$%
}\hbox{%
$2222223636322232$%
}%
}%
\hfill\copy\matricesbox
}%
\ifdim\wd\onelinebox>\myboxwidth
\hbox to \myboxwidth{%
$\Pi_{16,40}$ spans $L_{16.13}$%
\hfil
$2222223636322232$%
}%
\box\matricesbox
\else
\hbox to \myboxwidth{%
\unhbox\onelinebox
}%
\fi
\else
\hbox to \myboxwidth{%
$\Pi_{16,40}$ spans $L_{16.13}$%
\hfil}%
\hbox to \myboxwidth{%
$2222223636322232$%
\hfil}%
\box\matricesbox
\fi
}%
\hfill\discretionary{}{}{}%
\setbox\matricesbox=\hbox{%
{$\left[\!\llap{\phantom{%
\begingroup \smaller\smaller\smaller
\endgroup%
}}\!\right]$}%
}%
\ifdim\wd\matricesbox>\halfwidth\myboxwidth=\hsize\else\myboxwidth=\halfwidth\fi
\vbox{%
\ifdim\myboxwidth=\hsize
\setbox\onelinebox=\hbox{%
\vbox{\hbox{%
$\Pi_{16,41}$ spans $L_{16.13}$%
}\hbox{%
$2222232222363632$%
}%
}%
\hfill\copy\matricesbox
}%
\ifdim\wd\onelinebox>\myboxwidth
\hbox to \myboxwidth{%
$\Pi_{16,41}$ spans $L_{16.13}$%
\hfil
$2222232222363632$%
}%
\box\matricesbox
\else
\hbox to \myboxwidth{%
\unhbox\onelinebox
}%
\fi
\else
\hbox to \myboxwidth{%
$\Pi_{16,41}$ spans $L_{16.13}$%
\hfil}%
\hbox to \myboxwidth{%
$2222232222363632$%
\hfil}%
\box\matricesbox
\fi
}%
\hfill\discretionary{}{}{}%
\setbox\matricesbox=\hbox{%
{$\left[\!\llap{\phantom{%
\begingroup \smaller\smaller\smaller
\endgroup%
}}\!\right]$}%
}%
\ifdim\wd\matricesbox>\halfwidth\myboxwidth=\hsize\else\myboxwidth=\halfwidth\fi
\vbox{%
\ifdim\myboxwidth=\hsize
\setbox\onelinebox=\hbox{%
\vbox{\hbox{%
$\Pi_{16,42}$ spans $L_{16.13}$%
}\hbox{%
$2222232223223632$%
}%
}%
\hfill\copy\matricesbox
}%
\ifdim\wd\onelinebox>\myboxwidth
\hbox to \myboxwidth{%
$\Pi_{16,42}$ spans $L_{16.13}$%
\hfil
$2222232223223632$%
}%
\box\matricesbox
\else
\hbox to \myboxwidth{%
\unhbox\onelinebox
}%
\fi
\else
\hbox to \myboxwidth{%
$\Pi_{16,42}$ spans $L_{16.13}$%
\hfil}%
\hbox to \myboxwidth{%
$2222232223223632$%
\hfil}%
\box\matricesbox
\fi
}%
\hfill\discretionary{}{}{}%
\setbox\matricesbox=\hbox{%
{$\left[\!\llap{\phantom{%
\begingroup \smaller\smaller\smaller
\endgroup%
}}\!\right]$}%
}%
\ifdim\wd\matricesbox>\halfwidth\myboxwidth=\hsize\else\myboxwidth=\halfwidth\fi
\vbox{%
\ifdim\myboxwidth=\hsize
\setbox\onelinebox=\hbox{%
\vbox{\hbox{%
$\Pi_{16,43}$ spans $L_{16.13}$%
}\hbox{%
$2222232223632232$%
}%
}%
\hfill\copy\matricesbox
}%
\ifdim\wd\onelinebox>\myboxwidth
\hbox to \myboxwidth{%
$\Pi_{16,43}$ spans $L_{16.13}$%
\hfil
$2222232223632232$%
}%
\box\matricesbox
\else
\hbox to \myboxwidth{%
\unhbox\onelinebox
}%
\fi
\else
\hbox to \myboxwidth{%
$\Pi_{16,43}$ spans $L_{16.13}$%
\hfil}%
\hbox to \myboxwidth{%
$2222232223632232$%
\hfil}%
\box\matricesbox
\fi
}%
\hfill\discretionary{}{}{}%
\setbox\matricesbox=\hbox{%
{$\left[\!\llap{\phantom{%
\begingroup \smaller\smaller\smaller
\endgroup%
}}\!\right]$}%
}%
\ifdim\wd\matricesbox>\halfwidth\myboxwidth=\hsize\else\myboxwidth=\halfwidth\fi
\vbox{%
\ifdim\myboxwidth=\hsize
\setbox\onelinebox=\hbox{%
\vbox{\hbox{%
$\Pi_{16,44}$ spans $L_{16.13}$%
}\hbox{%
$2222232232223632$%
}%
}%
\hfill\copy\matricesbox
}%
\ifdim\wd\onelinebox>\myboxwidth
\hbox to \myboxwidth{%
$\Pi_{16,44}$ spans $L_{16.13}$%
\hfil
$2222232232223632$%
}%
\box\matricesbox
\else
\hbox to \myboxwidth{%
\unhbox\onelinebox
}%
\fi
\else
\hbox to \myboxwidth{%
$\Pi_{16,44}$ spans $L_{16.13}$%
\hfil}%
\hbox to \myboxwidth{%
$2222232232223632$%
\hfil}%
\box\matricesbox
\fi
}%
\hfill\discretionary{}{}{}%
\setbox\matricesbox=\hbox{%
{$\left[\!\llap{\phantom{%
\begingroup \smaller\smaller\smaller
\endgroup%
}}\!\right]$}%
}%
\ifdim\wd\matricesbox>\halfwidth\myboxwidth=\hsize\else\myboxwidth=\halfwidth\fi
\vbox{%
\ifdim\myboxwidth=\hsize
\setbox\onelinebox=\hbox{%
\vbox{\hbox{%
$\Pi_{16,45}$ spans $L_{16.13}$%
}\hbox{%
$3632222223222236$%
}%
}%
\hfill\copy\matricesbox
}%
\ifdim\wd\onelinebox>\myboxwidth
\hbox to \myboxwidth{%
$\Pi_{16,45}$ spans $L_{16.13}$%
\hfil
$3632222223222236$%
}%
\box\matricesbox
\else
\hbox to \myboxwidth{%
\unhbox\onelinebox
}%
\fi
\else
\hbox to \myboxwidth{%
$\Pi_{16,45}$ spans $L_{16.13}$%
\hfil}%
\hbox to \myboxwidth{%
$3632222223222236$%
\hfil}%
\box\matricesbox
\fi
}%
\hfill\discretionary{}{}{}%
\setbox\matricesbox=\hbox{%
{$\left[\!\llap{\phantom{%
\begingroup \smaller\smaller\smaller
\endgroup%
}}\!\right]$}%
}%
\ifdim\wd\matricesbox>\halfwidth\myboxwidth=\hsize\else\myboxwidth=\halfwidth\fi
\vbox{%
\ifdim\myboxwidth=\hsize
\setbox\onelinebox=\hbox{%
\vbox{\hbox{%
$\Pi_{16,46}$ spans $L_{16.13}$%
}\hbox{%
$3632222223222322$%
}%
}%
\hfill\copy\matricesbox
}%
\ifdim\wd\onelinebox>\myboxwidth
\hbox to \myboxwidth{%
$\Pi_{16,46}$ spans $L_{16.13}$%
\hfil
$3632222223222322$%
}%
\box\matricesbox
\else
\hbox to \myboxwidth{%
\unhbox\onelinebox
}%
\fi
\else
\hbox to \myboxwidth{%
$\Pi_{16,46}$ spans $L_{16.13}$%
\hfil}%
\hbox to \myboxwidth{%
$3632222223222322$%
\hfil}%
\box\matricesbox
\fi
}%
\hfill\discretionary{}{}{}%
\setbox\matricesbox=\hbox{%
{$\left[\!\llap{\phantom{%
\begingroup \smaller\smaller\smaller
\endgroup%
}}\!\right]$}%
}%
\ifdim\wd\matricesbox>\halfwidth\myboxwidth=\hsize\else\myboxwidth=\halfwidth\fi
\vbox{%
\ifdim\myboxwidth=\hsize
\setbox\onelinebox=\hbox{%
\vbox{\hbox{%
$\Pi_{16,47}$ spans $L_{16.13}$%
}\hbox{%
$3632222223223222$%
}%
}%
\hfill\copy\matricesbox
}%
\ifdim\wd\onelinebox>\myboxwidth
\hbox to \myboxwidth{%
$\Pi_{16,47}$ spans $L_{16.13}$%
\hfil
$3632222223223222$%
}%
\box\matricesbox
\else
\hbox to \myboxwidth{%
\unhbox\onelinebox
}%
\fi
\else
\hbox to \myboxwidth{%
$\Pi_{16,47}$ spans $L_{16.13}$%
\hfil}%
\hbox to \myboxwidth{%
$3632222223223222$%
\hfil}%
\box\matricesbox
\fi
}%
\hfill\discretionary{}{}{}%
\setbox\matricesbox=\hbox{%
{$\left[\!\llap{\phantom{%
\begingroup \smaller\smaller\smaller
\endgroup%
}}\!\right]$}%
}%
\ifdim\wd\matricesbox>\halfwidth\myboxwidth=\hsize\else\myboxwidth=\halfwidth\fi
\vbox{%
\ifdim\myboxwidth=\hsize
\setbox\onelinebox=\hbox{%
\vbox{\hbox{%
$\Pi_{16,48}$ spans $L_{16.13}$%
}\hbox{%
$3632222232222322$%
}%
}%
\hfill\copy\matricesbox
}%
\ifdim\wd\onelinebox>\myboxwidth
\hbox to \myboxwidth{%
$\Pi_{16,48}$ spans $L_{16.13}$%
\hfil
$3632222232222322$%
}%
\box\matricesbox
\else
\hbox to \myboxwidth{%
\unhbox\onelinebox
}%
\fi
\else
\hbox to \myboxwidth{%
$\Pi_{16,48}$ spans $L_{16.13}$%
\hfil}%
\hbox to \myboxwidth{%
$3632222232222322$%
\hfil}%
\box\matricesbox
\fi
}%
\hfill\discretionary{}{}{}%
\setbox\matricesbox=\hbox{%
{$\left[\!\llap{\phantom{%
\begingroup \smaller\smaller\smaller
\endgroup%
}}\!\right]$}%
}%
\ifdim\wd\matricesbox>\halfwidth\myboxwidth=\hsize\else\myboxwidth=\halfwidth\fi
\vbox{%
\ifdim\myboxwidth=\hsize
\setbox\onelinebox=\hbox{%
\vbox{\hbox{%
$\Pi_{16,49}$ spans $L_{16.13}$%
}\hbox{%
$3632222232223222$%
}%
}%
\hfill\copy\matricesbox
}%
\ifdim\wd\onelinebox>\myboxwidth
\hbox to \myboxwidth{%
$\Pi_{16,49}$ spans $L_{16.13}$%
\hfil
$3632222232223222$%
}%
\box\matricesbox
\else
\hbox to \myboxwidth{%
\unhbox\onelinebox
}%
\fi
\else
\hbox to \myboxwidth{%
$\Pi_{16,49}$ spans $L_{16.13}$%
\hfil}%
\hbox to \myboxwidth{%
$3632222232223222$%
\hfil}%
\box\matricesbox
\fi
}%
\hfill\discretionary{}{}{}%
\setbox\matricesbox=\hbox{%
{$\left[\!\llap{\phantom{%
\begingroup \smaller\smaller\smaller
\endgroup%
}}\!\right]$}%
}%
\ifdim\wd\matricesbox>\halfwidth\myboxwidth=\hsize\else\myboxwidth=\halfwidth\fi
\vbox{%
\ifdim\myboxwidth=\hsize
\setbox\onelinebox=\hbox{%
\vbox{\hbox{%
$\Pi_{16,50}$ spans $L_{16.13}$%
}\hbox{%
$3632222232232222$%
}%
}%
\hfill\copy\matricesbox
}%
\ifdim\wd\onelinebox>\myboxwidth
\hbox to \myboxwidth{%
$\Pi_{16,50}$ spans $L_{16.13}$%
\hfil
$3632222232232222$%
}%
\box\matricesbox
\else
\hbox to \myboxwidth{%
\unhbox\onelinebox
}%
\fi
\else
\hbox to \myboxwidth{%
$\Pi_{16,50}$ spans $L_{16.13}$%
\hfil}%
\hbox to \myboxwidth{%
$3632222232232222$%
\hfil}%
\box\matricesbox
\fi
}%
\hfill\discretionary{}{}{}%
\setbox\matricesbox=\hbox{%
{$\left[\!\llap{\phantom{%
\begingroup \smaller\smaller\smaller
\endgroup%
}}\!\right]$}%
}%
\ifdim\wd\matricesbox>\halfwidth\myboxwidth=\hsize\else\myboxwidth=\halfwidth\fi
\vbox{%
\ifdim\myboxwidth=\hsize
\setbox\onelinebox=\hbox{%
\vbox{\hbox{%
$\Pi_{16,51}$ spans $L_{16.13}$%
}\hbox{%
$3632222322222322$%
}%
}%
\hfill\copy\matricesbox
}%
\ifdim\wd\onelinebox>\myboxwidth
\hbox to \myboxwidth{%
$\Pi_{16,51}$ spans $L_{16.13}$%
\hfil
$3632222322222322$%
}%
\box\matricesbox
\else
\hbox to \myboxwidth{%
\unhbox\onelinebox
}%
\fi
\else
\hbox to \myboxwidth{%
$\Pi_{16,51}$ spans $L_{16.13}$%
\hfil}%
\hbox to \myboxwidth{%
$3632222322222322$%
\hfil}%
\box\matricesbox
\fi
}%
\hfill\discretionary{}{}{}%
\setbox\matricesbox=\hbox{%
{$\left[\!\llap{\phantom{%
\begingroup \smaller\smaller\smaller
\endgroup%
}}\!\right]$}%
}%
\ifdim\wd\matricesbox>\halfwidth\myboxwidth=\hsize\else\myboxwidth=\halfwidth\fi
\vbox{%
\ifdim\myboxwidth=\hsize
\setbox\onelinebox=\hbox{%
\vbox{\hbox{%
$\Pi_{16,52}$ spans $L_{16.13}$%
}\hbox{%
$3632222322223222$%
}%
}%
\hfill\copy\matricesbox
}%
\ifdim\wd\onelinebox>\myboxwidth
\hbox to \myboxwidth{%
$\Pi_{16,52}$ spans $L_{16.13}$%
\hfil
$3632222322223222$%
}%
\box\matricesbox
\else
\hbox to \myboxwidth{%
\unhbox\onelinebox
}%
\fi
\else
\hbox to \myboxwidth{%
$\Pi_{16,52}$ spans $L_{16.13}$%
\hfil}%
\hbox to \myboxwidth{%
$3632222322223222$%
\hfil}%
\box\matricesbox
\fi
}%
\hfill\discretionary{}{}{}%
\setbox\matricesbox=\hbox{%
{$\left[\!\llap{\phantom{%
\begingroup \smaller\smaller\smaller
\endgroup%
}}\!\right]$}%
}%
\ifdim\wd\matricesbox>\halfwidth\myboxwidth=\hsize\else\myboxwidth=\halfwidth\fi
\vbox{%
\ifdim\myboxwidth=\hsize
\setbox\onelinebox=\hbox{%
\vbox{\hbox{%
$\Pi_{16,53}$ spans $L_{16.13}$%
}\hbox{%
$3632222322232222$%
}%
}%
\hfill\copy\matricesbox
}%
\ifdim\wd\onelinebox>\myboxwidth
\hbox to \myboxwidth{%
$\Pi_{16,53}$ spans $L_{16.13}$%
\hfil
$3632222322232222$%
}%
\box\matricesbox
\else
\hbox to \myboxwidth{%
\unhbox\onelinebox
}%
\fi
\else
\hbox to \myboxwidth{%
$\Pi_{16,53}$ spans $L_{16.13}$%
\hfil}%
\hbox to \myboxwidth{%
$3632222322232222$%
\hfil}%
\box\matricesbox
\fi
}%
\hfill\discretionary{}{}{}%
\setbox\matricesbox=\hbox{%
{$\left[\!\llap{\phantom{%
\begingroup \smaller\smaller\smaller
\endgroup%
}}\!\right]$}%
}%
\ifdim\wd\matricesbox>\halfwidth\myboxwidth=\hsize\else\myboxwidth=\halfwidth\fi
\vbox{%
\ifdim\myboxwidth=\hsize
\setbox\onelinebox=\hbox{%
\vbox{\hbox{%
$\Pi_{16,54}$ spans $L_{16.13}$%
}\hbox{%
$3632222322322222$%
}%
}%
\hfill\copy\matricesbox
}%
\ifdim\wd\onelinebox>\myboxwidth
\hbox to \myboxwidth{%
$\Pi_{16,54}$ spans $L_{16.13}$%
\hfil
$3632222322322222$%
}%
\box\matricesbox
\else
\hbox to \myboxwidth{%
\unhbox\onelinebox
}%
\fi
\else
\hbox to \myboxwidth{%
$\Pi_{16,54}$ spans $L_{16.13}$%
\hfil}%
\hbox to \myboxwidth{%
$3632222322322222$%
\hfil}%
\box\matricesbox
\fi
}%
\hfill\discretionary{}{}{}%
\setbox\matricesbox=\hbox{%
{$\left[\!\llap{\phantom{%
\begingroup \smaller\smaller\smaller
\endgroup%
}}\!\right]$}%
}%
\ifdim\wd\matricesbox>\halfwidth\myboxwidth=\hsize\else\myboxwidth=\halfwidth\fi
\vbox{%
\ifdim\myboxwidth=\hsize
\setbox\onelinebox=\hbox{%
\vbox{\hbox{%
$\Pi_{16,55}$ spans $L_{16.13}$%
}\hbox{%
$3632223222223222$%
}%
}%
\hfill\copy\matricesbox
}%
\ifdim\wd\onelinebox>\myboxwidth
\hbox to \myboxwidth{%
$\Pi_{16,55}$ spans $L_{16.13}$%
\hfil
$3632223222223222$%
}%
\box\matricesbox
\else
\hbox to \myboxwidth{%
\unhbox\onelinebox
}%
\fi
\else
\hbox to \myboxwidth{%
$\Pi_{16,55}$ spans $L_{16.13}$%
\hfil}%
\hbox to \myboxwidth{%
$3632223222223222$%
\hfil}%
\box\matricesbox
\fi
}%
\hfill\discretionary{}{}{}%
\setbox\matricesbox=\hbox{%
{$\left[\!\llap{\phantom{%
\begingroup \smaller\smaller\smaller
\endgroup%
}}\!\right]$}%
}%
\ifdim\wd\matricesbox>\halfwidth\myboxwidth=\hsize\else\myboxwidth=\halfwidth\fi
\vbox{%
\ifdim\myboxwidth=\hsize
\setbox\onelinebox=\hbox{%
\vbox{\hbox{%
$\Pi_{16,56}$ spans $L_{16.13}$%
}\hbox{%
$3632223222232222$%
}%
}%
\hfill\copy\matricesbox
}%
\ifdim\wd\onelinebox>\myboxwidth
\hbox to \myboxwidth{%
$\Pi_{16,56}$ spans $L_{16.13}$%
\hfil
$3632223222232222$%
}%
\box\matricesbox
\else
\hbox to \myboxwidth{%
\unhbox\onelinebox
}%
\fi
\else
\hbox to \myboxwidth{%
$\Pi_{16,56}$ spans $L_{16.13}$%
\hfil}%
\hbox to \myboxwidth{%
$3632223222232222$%
\hfil}%
\box\matricesbox
\fi
}%
\hfill\discretionary{}{}{}%
\setbox\matricesbox=\hbox{%
{$\left[\!\llap{\phantom{%
\begingroup \smaller\smaller\smaller
\endgroup%
}}\!\right]$}%
}%
\ifdim\wd\matricesbox>\halfwidth\myboxwidth=\hsize\else\myboxwidth=\halfwidth\fi
\vbox{%
\ifdim\myboxwidth=\hsize
\setbox\onelinebox=\hbox{%
\vbox{\hbox{%
$\Pi_{16,57}$ spans $L_{16.13}$%
}\hbox{%
$3632223222322222$%
}%
}%
\hfill\copy\matricesbox
}%
\ifdim\wd\onelinebox>\myboxwidth
\hbox to \myboxwidth{%
$\Pi_{16,57}$ spans $L_{16.13}$%
\hfil
$3632223222322222$%
}%
\box\matricesbox
\else
\hbox to \myboxwidth{%
\unhbox\onelinebox
}%
\fi
\else
\hbox to \myboxwidth{%
$\Pi_{16,57}$ spans $L_{16.13}$%
\hfil}%
\hbox to \myboxwidth{%
$3632223222322222$%
\hfil}%
\box\matricesbox
\fi
}%
\hfill\discretionary{}{}{}%
\setbox\matricesbox=\hbox{%
{$\left[\!\llap{\phantom{%
\begingroup \smaller\smaller\smaller
\endgroup%
}}\!\right]$}%
}%
\ifdim\wd\matricesbox>\halfwidth\myboxwidth=\hsize\else\myboxwidth=\halfwidth\fi
\vbox{%
\ifdim\myboxwidth=\hsize
\setbox\onelinebox=\hbox{%
\vbox{\hbox{%
$\Pi_{16,58}$ spans $L_{16.13}$%
}\hbox{%
$3632232222223222$%
}%
}%
\hfill\copy\matricesbox
}%
\ifdim\wd\onelinebox>\myboxwidth
\hbox to \myboxwidth{%
$\Pi_{16,58}$ spans $L_{16.13}$%
\hfil
$3632232222223222$%
}%
\box\matricesbox
\else
\hbox to \myboxwidth{%
\unhbox\onelinebox
}%
\fi
\else
\hbox to \myboxwidth{%
$\Pi_{16,58}$ spans $L_{16.13}$%
\hfil}%
\hbox to \myboxwidth{%
$3632232222223222$%
\hfil}%
\box\matricesbox
\fi
}%
\hfill\discretionary{}{}{}%
\setbox\matricesbox=\hbox{%
{$\left[\!\llap{\phantom{%
\begingroup \smaller\smaller\smaller
\endgroup%
}}\!\right]$}%
}%
\ifdim\wd\matricesbox>\halfwidth\myboxwidth=\hsize\else\myboxwidth=\halfwidth\fi
\vbox{%
\ifdim\myboxwidth=\hsize
\setbox\onelinebox=\hbox{%
\vbox{\hbox{%
$\Pi_{16,59}$ spans $L_{16.13}$%
}\hbox{%
$3632232222232222$%
}%
}%
\hfill\copy\matricesbox
}%
\ifdim\wd\onelinebox>\myboxwidth
\hbox to \myboxwidth{%
$\Pi_{16,59}$ spans $L_{16.13}$%
\hfil
$3632232222232222$%
}%
\box\matricesbox
\else
\hbox to \myboxwidth{%
\unhbox\onelinebox
}%
\fi
\else
\hbox to \myboxwidth{%
$\Pi_{16,59}$ spans $L_{16.13}$%
\hfil}%
\hbox to \myboxwidth{%
$3632232222232222$%
\hfil}%
\box\matricesbox
\fi
}%
\hfill\discretionary{}{}{}%
\setbox\matricesbox=\hbox{%
{$\left[\!\llap{\phantom{%
\begingroup \smaller\smaller\smaller
\endgroup%
}}\!\right]$}%
}%
\ifdim\wd\matricesbox>\halfwidth\myboxwidth=\hsize\else\myboxwidth=\halfwidth\fi
\vbox{%
\ifdim\myboxwidth=\hsize
\setbox\onelinebox=\hbox{%
\vbox{\hbox{%
$\Pi_{16,60}$ spans $L_{16.13}$%
}\hbox{%
$3632232222322222$%
}%
}%
\hfill\copy\matricesbox
}%
\ifdim\wd\onelinebox>\myboxwidth
\hbox to \myboxwidth{%
$\Pi_{16,60}$ spans $L_{16.13}$%
\hfil
$3632232222322222$%
}%
\box\matricesbox
\else
\hbox to \myboxwidth{%
\unhbox\onelinebox
}%
\fi
\else
\hbox to \myboxwidth{%
$\Pi_{16,60}$ spans $L_{16.13}$%
\hfil}%
\hbox to \myboxwidth{%
$3632232222322222$%
\hfil}%
\box\matricesbox
\fi
}%
\hfill\discretionary{}{}{}%
\setbox\matricesbox=\hbox{%
{$\left[\!\llap{\phantom{%
\begingroup \smaller\smaller\smaller
\endgroup%
}}\!\right]$}%
}%
\ifdim\wd\matricesbox>\halfwidth\myboxwidth=\hsize\else\myboxwidth=\halfwidth\fi
\vbox{%
\ifdim\myboxwidth=\hsize
\setbox\onelinebox=\hbox{%
\vbox{\hbox{%
$\Pi_{16,61}$ spans $L_{142.20}$%
}\hbox{%
$\infty422\infty22\infty4224\infty422$%
}%
}%
\hfill\copy\matricesbox
}%
\ifdim\wd\onelinebox>\myboxwidth
\hbox to \myboxwidth{%
$\Pi_{16,61}$ spans $L_{142.20}$%
\hfil
$\infty422\infty22\infty4224\infty422$%
}%
\box\matricesbox
\else
\hbox to \myboxwidth{%
\unhbox\onelinebox
}%
\fi
\else
\hbox to \myboxwidth{%
$\Pi_{16,61}$ spans $L_{142.20}$%
\hfil}%
\hbox to \myboxwidth{%
$\infty422\infty22\infty4224\infty422$%
\hfil}%
\box\matricesbox
\fi
}%
\hfill\discretionary{}{}{}%
\setbox\matricesbox=\hbox{%
{$\left[\!\llap{\phantom{%
\begingroup \smaller\smaller\smaller
\endgroup%
}}\!\right]$}%
}%
\ifdim\wd\matricesbox>\halfwidth\myboxwidth=\hsize\else\myboxwidth=\halfwidth\fi
\vbox{%
\ifdim\myboxwidth=\hsize
\setbox\onelinebox=\hbox{%
\vbox{\hbox{%
$\Pi_{16,62}$ spans $L_{142.20}$%
}\hbox{%
$\infty422\infty224\infty422\infty422$%
}%
}%
\hfill\copy\matricesbox
}%
\ifdim\wd\onelinebox>\myboxwidth
\hbox to \myboxwidth{%
$\Pi_{16,62}$ spans $L_{142.20}$%
\hfil
$\infty422\infty224\infty422\infty422$%
}%
\box\matricesbox
\else
\hbox to \myboxwidth{%
\unhbox\onelinebox
}%
\fi
\else
\hbox to \myboxwidth{%
$\Pi_{16,62}$ spans $L_{142.20}$%
\hfil}%
\hbox to \myboxwidth{%
$\infty422\infty224\infty422\infty422$%
\hfil}%
\box\matricesbox
\fi
}%
\hfill\discretionary{}{}{}%
\setbox\matricesbox=\hbox{%
{$\left[\!\llap{\phantom{%
\begingroup \smaller\smaller\smaller
\endgroup%
}}\!\right]$}%
}%
\ifdim\wd\matricesbox>\halfwidth\myboxwidth=\hsize\else\myboxwidth=\halfwidth\fi
\vbox{%
\ifdim\myboxwidth=\hsize
\setbox\onelinebox=\hbox{%
\vbox{\hbox{%
$\Pi_{16,63}$ spans $L_{142.20}$%
}\hbox{%
$\infty422\infty224\infty4224\infty22$%
}%
}%
\hfill\copy\matricesbox
}%
\ifdim\wd\onelinebox>\myboxwidth
\hbox to \myboxwidth{%
$\Pi_{16,63}$ spans $L_{142.20}$%
\hfil
$\infty422\infty224\infty4224\infty22$%
}%
\box\matricesbox
\else
\hbox to \myboxwidth{%
\unhbox\onelinebox
}%
\fi
\else
\hbox to \myboxwidth{%
$\Pi_{16,63}$ spans $L_{142.20}$%
\hfil}%
\hbox to \myboxwidth{%
$\infty422\infty224\infty4224\infty22$%
\hfil}%
\box\matricesbox
\fi
}%
\hfill\discretionary{}{}{}%
\setbox\matricesbox=\hbox{%
{$\left[\!\llap{\phantom{%
\begingroup \smaller\smaller\smaller
\endgroup%
}}\!\right]$}%
}%
\ifdim\wd\matricesbox>\halfwidth\myboxwidth=\hsize\else\myboxwidth=\halfwidth\fi
\vbox{%
\ifdim\myboxwidth=\hsize
\setbox\onelinebox=\hbox{%
\vbox{\hbox{%
$\Pi_{16,64}$ spans $L_{142.20}$%
}\hbox{%
$\infty422\infty422\infty4224\infty22$%
}%
}%
\hfill\copy\matricesbox
}%
\ifdim\wd\onelinebox>\myboxwidth
\hbox to \myboxwidth{%
$\Pi_{16,64}$ spans $L_{142.20}$%
\hfil
$\infty422\infty422\infty4224\infty22$%
}%
\box\matricesbox
\else
\hbox to \myboxwidth{%
\unhbox\onelinebox
}%
\fi
\else
\hbox to \myboxwidth{%
$\Pi_{16,64}$ spans $L_{142.20}$%
\hfil}%
\hbox to \myboxwidth{%
$\infty422\infty422\infty4224\infty22$%
\hfil}%
\box\matricesbox
\fi
}%
\hfill\discretionary{}{}{}%
\setbox\matricesbox=\hbox{%
{$\left[\!\llap{\phantom{%
\begingroup \smaller\smaller\smaller
\endgroup%
}}\!\right]$}%
}%
\ifdim\wd\matricesbox>\halfwidth\myboxwidth=\hsize\else\myboxwidth=\halfwidth\fi
\vbox{%
\ifdim\myboxwidth=\hsize
\setbox\onelinebox=\hbox{%
\vbox{\hbox{%
$\Pi_{16,65}$ spans $L_{142.20}$%
}\hbox{%
$\infty422\infty4224\infty422\infty22$%
}%
}%
\hfill\copy\matricesbox
}%
\ifdim\wd\onelinebox>\myboxwidth
\hbox to \myboxwidth{%
$\Pi_{16,65}$ spans $L_{142.20}$%
\hfil
$\infty422\infty4224\infty422\infty22$%
}%
\box\matricesbox
\else
\hbox to \myboxwidth{%
\unhbox\onelinebox
}%
\fi
\else
\hbox to \myboxwidth{%
$\Pi_{16,65}$ spans $L_{142.20}$%
\hfil}%
\hbox to \myboxwidth{%
$\infty422\infty4224\infty422\infty22$%
\hfil}%
\box\matricesbox
\fi
}%
\hfill\discretionary{}{}{}%
\setbox\matricesbox=\hbox{%
{$\left[\!\llap{\phantom{%
\begingroup \smaller\smaller\smaller
\endgroup%
}}\!\right]$}%
}%
\ifdim\wd\matricesbox>\halfwidth\myboxwidth=\hsize\else\myboxwidth=\halfwidth\fi
\vbox{%
\ifdim\myboxwidth=\hsize
\setbox\onelinebox=\hbox{%
\vbox{\hbox{%
$\Pi_{16,66}$ spans $L_{142.20}$%
}\hbox{%
$\infty4224\infty22\infty224\infty422$%
}%
}%
\hfill\copy\matricesbox
}%
\ifdim\wd\onelinebox>\myboxwidth
\hbox to \myboxwidth{%
$\Pi_{16,66}$ spans $L_{142.20}$%
\hfil
$\infty4224\infty22\infty224\infty422$%
}%
\box\matricesbox
\else
\hbox to \myboxwidth{%
\unhbox\onelinebox
}%
\fi
\else
\hbox to \myboxwidth{%
$\Pi_{16,66}$ spans $L_{142.20}$%
\hfil}%
\hbox to \myboxwidth{%
$\infty4224\infty22\infty224\infty422$%
\hfil}%
\box\matricesbox
\fi
}%
\hfill\discretionary{}{}{}%
\setbox\matricesbox=\hbox{%
{$\left[\!\llap{\phantom{%
\begingroup \smaller\smaller\smaller
\endgroup%
}}\!\right]$}%
}%
\ifdim\wd\matricesbox>\halfwidth\myboxwidth=\hsize\else\myboxwidth=\halfwidth\fi
\vbox{%
\ifdim\myboxwidth=\hsize
\setbox\onelinebox=\hbox{%
\vbox{\hbox{%
$\Pi_{16,67}$ spans $L_{16.13}$%
}\hbox{%
$3222223222322322$%
}%
}%
\hfill\copy\matricesbox
}%
\ifdim\wd\onelinebox>\myboxwidth
\hbox to \myboxwidth{%
$\Pi_{16,67}$ spans $L_{16.13}$%
\hfil
$3222223222322322$%
}%
\box\matricesbox
\else
\hbox to \myboxwidth{%
\unhbox\onelinebox
}%
\fi
\else
\hbox to \myboxwidth{%
$\Pi_{16,67}$ spans $L_{16.13}$%
\hfil}%
\hbox to \myboxwidth{%
$3222223222322322$%
\hfil}%
\box\matricesbox
\fi
}%
\hfill\discretionary{}{}{}%
\setbox\matricesbox=\hbox{%
{$\left[\!\llap{\phantom{%
\begingroup \smaller\smaller\smaller
\endgroup%
}}\!\right]$}%
}%
\ifdim\wd\matricesbox>\halfwidth\myboxwidth=\hsize\else\myboxwidth=\halfwidth\fi
\vbox{%
\ifdim\myboxwidth=\hsize
\setbox\onelinebox=\hbox{%
\vbox{\hbox{%
$\Pi_{16,68}$ spans $L_{16.13}$%
}\hbox{%
$3222223223222322$%
}%
}%
\hfill\copy\matricesbox
}%
\ifdim\wd\onelinebox>\myboxwidth
\hbox to \myboxwidth{%
$\Pi_{16,68}$ spans $L_{16.13}$%
\hfil
$3222223223222322$%
}%
\box\matricesbox
\else
\hbox to \myboxwidth{%
\unhbox\onelinebox
}%
\fi
\else
\hbox to \myboxwidth{%
$\Pi_{16,68}$ spans $L_{16.13}$%
\hfil}%
\hbox to \myboxwidth{%
$3222223223222322$%
\hfil}%
\box\matricesbox
\fi
}%
\hfill\discretionary{}{}{}%
\setbox\matricesbox=\hbox{%
{$\left[\!\llap{\phantom{%
\begingroup \smaller\smaller\smaller
\endgroup%
}}\!\right]$}%
}%
\ifdim\wd\matricesbox>\halfwidth\myboxwidth=\hsize\else\myboxwidth=\halfwidth\fi
\vbox{%
\ifdim\myboxwidth=\hsize
\setbox\onelinebox=\hbox{%
\vbox{\hbox{%
$\Pi_{16,69}$ spans $L_{16.13}$%
}\hbox{%
$3222223223223222$%
}%
}%
\hfill\copy\matricesbox
}%
\ifdim\wd\onelinebox>\myboxwidth
\hbox to \myboxwidth{%
$\Pi_{16,69}$ spans $L_{16.13}$%
\hfil
$3222223223223222$%
}%
\box\matricesbox
\else
\hbox to \myboxwidth{%
\unhbox\onelinebox
}%
\fi
\else
\hbox to \myboxwidth{%
$\Pi_{16,69}$ spans $L_{16.13}$%
\hfil}%
\hbox to \myboxwidth{%
$3222223223223222$%
\hfil}%
\box\matricesbox
\fi
}%
\hfill\discretionary{}{}{}%
\setbox\matricesbox=\hbox{%
{$\left[\!\llap{\phantom{%
\begingroup \smaller\smaller\smaller
\endgroup%
}}\!\right]$}%
}%
\ifdim\wd\matricesbox>\halfwidth\myboxwidth=\hsize\else\myboxwidth=\halfwidth\fi
\vbox{%
\ifdim\myboxwidth=\hsize
\setbox\onelinebox=\hbox{%
\vbox{\hbox{%
$\Pi_{16,70}$ spans $L_{16.13}$%
}\hbox{%
$3222232223222322$%
}%
}%
\hfill\copy\matricesbox
}%
\ifdim\wd\onelinebox>\myboxwidth
\hbox to \myboxwidth{%
$\Pi_{16,70}$ spans $L_{16.13}$%
\hfil
$3222232223222322$%
}%
\box\matricesbox
\else
\hbox to \myboxwidth{%
\unhbox\onelinebox
}%
\fi
\else
\hbox to \myboxwidth{%
$\Pi_{16,70}$ spans $L_{16.13}$%
\hfil}%
\hbox to \myboxwidth{%
$3222232223222322$%
\hfil}%
\box\matricesbox
\fi
}%
\hfill\discretionary{}{}{}%
\setbox\matricesbox=\hbox{%
{$\left[\!\llap{\phantom{%
\begingroup \smaller\smaller\smaller
\endgroup%
}}\!\right]$}%
}%
\ifdim\wd\matricesbox>\halfwidth\myboxwidth=\hsize\else\myboxwidth=\halfwidth\fi
\vbox{%
\ifdim\myboxwidth=\hsize
\setbox\onelinebox=\hbox{%
\vbox{\hbox{%
$\Pi_{16,71}$ spans $L_{16.13}$%
}\hbox{%
$3222322223223222$%
}%
}%
\hfill\copy\matricesbox
}%
\ifdim\wd\onelinebox>\myboxwidth
\hbox to \myboxwidth{%
$\Pi_{16,71}$ spans $L_{16.13}$%
\hfil
$3222322223223222$%
}%
\box\matricesbox
\else
\hbox to \myboxwidth{%
\unhbox\onelinebox
}%
\fi
\else
\hbox to \myboxwidth{%
$\Pi_{16,71}$ spans $L_{16.13}$%
\hfil}%
\hbox to \myboxwidth{%
$3222322223223222$%
\hfil}%
\box\matricesbox
\fi
}%
\hfill\discretionary{}{}{}%

\vskip2pt\hrule\vskip2pt

\leavevmode\setbox\matricesbox=\hbox{%
{$\left[\!\llap{\phantom{%
\begingroup \smaller\smaller\smaller\begin{tabular}{@{}c@{}}%
\phantom{0}\\\phantom{0}\\\phantom{0}
\end{tabular}\endgroup%
}}\right.$}%
\begingroup \smaller\smaller\smaller\begin{tabular}{@{}c@{}}%
-1\\\phantom{0}\\\phantom{0}
\end{tabular}\endgroup%
\kern3pt%
\begingroup \smaller\smaller\smaller\begin{tabular}{@{}c@{}}%
\phantom{0}\\45/2\\\phantom{0}
\end{tabular}\endgroup%
\kern3pt%
\begingroup \smaller\smaller\smaller\begin{tabular}{@{}c@{}}%
\phantom{0}\\\phantom{0}\\15/2
\end{tabular}\endgroup%
{$\left.\llap{\phantom{%
\begingroup \smaller\smaller\smaller\begin{tabular}{@{}c@{}}%
\phantom{0}\\\phantom{0}\\\phantom{0}
\end{tabular}\endgroup%
}}\!\right]$}%
{$\left[\!\llap{\phantom{%
\begingroup \smaller\smaller\smaller\begin{tabular}{@{}c@{}}%
0\\0\\0
\end{tabular}\endgroup%
}}\right.$}%
\begingroup \smaller\smaller\smaller\begin{tabular}{@{}c@{}}%
9\\2\\0
\end{tabular}\endgroup%
\kern3pt%
\begingroup \smaller\smaller\smaller\begin{tabular}{@{}c@{}}%
5\\1\\1
\end{tabular}\endgroup%
\kern3pt%
\begingroup \smaller\smaller\smaller\begin{tabular}{@{}c@{}}%
9\\1\\3
\end{tabular}\endgroup%
\kern3pt%
\begingroup \smaller\smaller\smaller\begin{tabular}{@{}c@{}}%
5\\0\\2
\end{tabular}\endgroup%
\kern3pt%
\begingroup \smaller\smaller\smaller\begin{tabular}{@{}c@{}}%
90\\-8\\30
\end{tabular}\endgroup%
\kern3pt%
\begingroup \smaller\smaller\smaller\begin{tabular}{@{}c@{}}%
90\\-11\\27
\end{tabular}\endgroup%
\kern3pt%
\begingroup \smaller\smaller\smaller\begin{tabular}{@{}c@{}}%
30\\-5\\7
\end{tabular}\endgroup%
\kern3pt%
\begingroup \smaller\smaller\smaller\begin{tabular}{@{}c@{}}%
30\\-6\\4
\end{tabular}\endgroup%
\kern3pt%
\begingroup \smaller\smaller\smaller\begin{tabular}{@{}c@{}}%
90\\-19\\3
\end{tabular}\endgroup%
{$\left.\llap{\phantom{%
\begingroup \smaller\smaller\smaller\begin{tabular}{@{}c@{}}%
0\\0\\0
\end{tabular}\endgroup%
}}\!\right]$}%
}%
\ifdim\wd\matricesbox>\halfwidth\myboxwidth=\hsize\else\myboxwidth=\halfwidth\fi
\vbox{%
\ifdim\myboxwidth=\hsize
\setbox\onelinebox=\hbox{%
\vbox{\hbox{%
$\Pi_{17,1}$ spans $L_{16.13}$%
}\hbox{%
$222|22223636\slashthree63632\rtimes D_{2}$%
}%
}%
\hfill\copy\matricesbox
}%
\ifdim\wd\onelinebox>\myboxwidth
\hbox to \myboxwidth{%
$\Pi_{17,1}$ spans $L_{16.13}$%
\hfil
$222|22223636\slashthree63632\rtimes D_{2}$%
}%
\box\matricesbox
\else
\hbox to \myboxwidth{%
\unhbox\onelinebox
}%
\fi
\else
\hbox to \myboxwidth{%
$\Pi_{17,1}$ spans $L_{16.13}$%
\hfil}%
\hbox to \myboxwidth{%
$222|22223636\slashthree63632\rtimes D_{2}$%
\hfil}%
\box\matricesbox
\fi
}%
\hfill\discretionary{}{}{}%
\setbox\matricesbox=\hbox{%
{$\left[\!\llap{\phantom{%
\begingroup \smaller\smaller\smaller\begin{tabular}{@{}c@{}}%
\phantom{0}\\\phantom{0}\\\phantom{0}
\end{tabular}\endgroup%
}}\right.$}%
\begingroup \smaller\smaller\smaller\begin{tabular}{@{}c@{}}%
-1\\\phantom{0}\\\phantom{0}
\end{tabular}\endgroup%
\kern3pt%
\begingroup \smaller\smaller\smaller\begin{tabular}{@{}c@{}}%
\phantom{0}\\15/2\\\phantom{0}
\end{tabular}\endgroup%
\kern3pt%
\begingroup \smaller\smaller\smaller\begin{tabular}{@{}c@{}}%
\phantom{0}\\\phantom{0}\\45/2
\end{tabular}\endgroup%
{$\left.\llap{\phantom{%
\begingroup \smaller\smaller\smaller\begin{tabular}{@{}c@{}}%
\phantom{0}\\\phantom{0}\\\phantom{0}
\end{tabular}\endgroup%
}}\!\right]$}%
{$\left[\!\llap{\phantom{%
\begingroup \smaller\smaller\smaller\begin{tabular}{@{}c@{}}%
0\\0\\0
\end{tabular}\endgroup%
}}\right.$}%
\begingroup \smaller\smaller\smaller\begin{tabular}{@{}c@{}}%
5\\-2\\0
\end{tabular}\endgroup%
\kern3pt%
\begingroup \smaller\smaller\smaller\begin{tabular}{@{}c@{}}%
9\\-3\\1
\end{tabular}\endgroup%
\kern3pt%
\begingroup \smaller\smaller\smaller\begin{tabular}{@{}c@{}}%
5\\-1\\1
\end{tabular}\endgroup%
\kern3pt%
\begingroup \smaller\smaller\smaller\begin{tabular}{@{}c@{}}%
9\\0\\2
\end{tabular}\endgroup%
\kern3pt%
\begingroup \smaller\smaller\smaller\begin{tabular}{@{}c@{}}%
30\\4\\6
\end{tabular}\endgroup%
\kern3pt%
\begingroup \smaller\smaller\smaller\begin{tabular}{@{}c@{}}%
30\\7\\5
\end{tabular}\endgroup%
\kern3pt%
\begingroup \smaller\smaller\smaller\begin{tabular}{@{}c@{}}%
90\\27\\11
\end{tabular}\endgroup%
\kern3pt%
\begingroup \smaller\smaller\smaller\begin{tabular}{@{}c@{}}%
90\\30\\8
\end{tabular}\endgroup%
\kern3pt%
\begingroup \smaller\smaller\smaller\begin{tabular}{@{}c@{}}%
30\\11\\1
\end{tabular}\endgroup%
{$\left.\llap{\phantom{%
\begingroup \smaller\smaller\smaller\begin{tabular}{@{}c@{}}%
0\\0\\0
\end{tabular}\endgroup%
}}\!\right]$}%
}%
\ifdim\wd\matricesbox>\halfwidth\myboxwidth=\hsize\else\myboxwidth=\halfwidth\fi
\vbox{%
\ifdim\myboxwidth=\hsize
\setbox\onelinebox=\hbox{%
\vbox{\hbox{%
$\Pi_{17,2}$ spans $L_{16.13}$%
}\hbox{%
$22|22223636\slashthree636322\rtimes D_{2}$%
}%
}%
\hfill\copy\matricesbox
}%
\ifdim\wd\onelinebox>\myboxwidth
\hbox to \myboxwidth{%
$\Pi_{17,2}$ spans $L_{16.13}$%
\hfil
$22|22223636\slashthree636322\rtimes D_{2}$%
}%
\box\matricesbox
\else
\hbox to \myboxwidth{%
\unhbox\onelinebox
}%
\fi
\else
\hbox to \myboxwidth{%
$\Pi_{17,2}$ spans $L_{16.13}$%
\hfil}%
\hbox to \myboxwidth{%
$22|22223636\slashthree636322\rtimes D_{2}$%
\hfil}%
\box\matricesbox
\fi
}%
\hfill\discretionary{}{}{}%
\setbox\matricesbox=\hbox{%
{$\left[\!\llap{\phantom{%
\begingroup \smaller\smaller\smaller\begin{tabular}{@{}c@{}}%
\phantom{0}\\\phantom{0}\\\phantom{0}
\end{tabular}\endgroup%
}}\right.$}%
\begingroup \smaller\smaller\smaller\begin{tabular}{@{}c@{}}%
-1\\\phantom{0}\\\phantom{0}
\end{tabular}\endgroup%
\kern3pt%
\begingroup \smaller\smaller\smaller\begin{tabular}{@{}c@{}}%
\phantom{0}\\15/2\\\phantom{0}
\end{tabular}\endgroup%
\kern3pt%
\begingroup \smaller\smaller\smaller\begin{tabular}{@{}c@{}}%
\phantom{0}\\\phantom{0}\\45/2
\end{tabular}\endgroup%
{$\left.\llap{\phantom{%
\begingroup \smaller\smaller\smaller\begin{tabular}{@{}c@{}}%
\phantom{0}\\\phantom{0}\\\phantom{0}
\end{tabular}\endgroup%
}}\!\right]$}%
{$\left[\!\llap{\phantom{%
\begingroup \smaller\smaller\smaller\begin{tabular}{@{}c@{}}%
0\\0\\0
\end{tabular}\endgroup%
}}\right.$}%
\begingroup \smaller\smaller\smaller\begin{tabular}{@{}c@{}}%
5\\-2\\0
\end{tabular}\endgroup%
\kern3pt%
\begingroup \smaller\smaller\smaller\begin{tabular}{@{}c@{}}%
9\\-3\\1
\end{tabular}\endgroup%
\kern3pt%
\begingroup \smaller\smaller\smaller\begin{tabular}{@{}c@{}}%
5\\-1\\1
\end{tabular}\endgroup%
\kern3pt%
\begingroup \smaller\smaller\smaller\begin{tabular}{@{}c@{}}%
90\\-3\\19
\end{tabular}\endgroup%
\kern3pt%
\begingroup \smaller\smaller\smaller\begin{tabular}{@{}c@{}}%
90\\3\\19
\end{tabular}\endgroup%
\kern3pt%
\begingroup \smaller\smaller\smaller\begin{tabular}{@{}c@{}}%
5\\1\\1
\end{tabular}\endgroup%
\kern3pt%
\begingroup \smaller\smaller\smaller\begin{tabular}{@{}c@{}}%
90\\27\\11
\end{tabular}\endgroup%
\kern3pt%
\begingroup \smaller\smaller\smaller\begin{tabular}{@{}c@{}}%
90\\30\\8
\end{tabular}\endgroup%
\kern3pt%
\begingroup \smaller\smaller\smaller\begin{tabular}{@{}c@{}}%
30\\11\\1
\end{tabular}\endgroup%
{$\left.\llap{\phantom{%
\begingroup \smaller\smaller\smaller\begin{tabular}{@{}c@{}}%
0\\0\\0
\end{tabular}\endgroup%
}}\!\right]$}%
}%
\ifdim\wd\matricesbox>\halfwidth\myboxwidth=\hsize\else\myboxwidth=\halfwidth\fi
\vbox{%
\ifdim\myboxwidth=\hsize
\setbox\onelinebox=\hbox{%
\vbox{\hbox{%
$\Pi_{17,3}$ spans $L_{16.13}$%
}\hbox{%
$22|22232236\slashthree632232\rtimes D_{2}$%
}%
}%
\hfill\copy\matricesbox
}%
\ifdim\wd\onelinebox>\myboxwidth
\hbox to \myboxwidth{%
$\Pi_{17,3}$ spans $L_{16.13}$%
\hfil
$22|22232236\slashthree632232\rtimes D_{2}$%
}%
\box\matricesbox
\else
\hbox to \myboxwidth{%
\unhbox\onelinebox
}%
\fi
\else
\hbox to \myboxwidth{%
$\Pi_{17,3}$ spans $L_{16.13}$%
\hfil}%
\hbox to \myboxwidth{%
$22|22232236\slashthree632232\rtimes D_{2}$%
\hfil}%
\box\matricesbox
\fi
}%
\hfill\discretionary{}{}{}%
\setbox\matricesbox=\hbox{%
{$\left[\!\llap{\phantom{%
\begingroup \smaller\smaller\smaller\begin{tabular}{@{}c@{}}%
\phantom{0}\\\phantom{0}\\\phantom{0}
\end{tabular}\endgroup%
}}\right.$}%
\begingroup \smaller\smaller\smaller\begin{tabular}{@{}c@{}}%
-1\\\phantom{0}\\\phantom{0}
\end{tabular}\endgroup%
\kern3pt%
\begingroup \smaller\smaller\smaller\begin{tabular}{@{}c@{}}%
\phantom{0}\\15/2\\\phantom{0}
\end{tabular}\endgroup%
\kern3pt%
\begingroup \smaller\smaller\smaller\begin{tabular}{@{}c@{}}%
\phantom{0}\\\phantom{0}\\45/2
\end{tabular}\endgroup%
{$\left.\llap{\phantom{%
\begingroup \smaller\smaller\smaller\begin{tabular}{@{}c@{}}%
\phantom{0}\\\phantom{0}\\\phantom{0}
\end{tabular}\endgroup%
}}\!\right]$}%
{$\left[\!\llap{\phantom{%
\begingroup \smaller\smaller\smaller\begin{tabular}{@{}c@{}}%
0\\0\\0
\end{tabular}\endgroup%
}}\right.$}%
\begingroup \smaller\smaller\smaller\begin{tabular}{@{}c@{}}%
5\\-2\\0
\end{tabular}\endgroup%
\kern3pt%
\begingroup \smaller\smaller\smaller\begin{tabular}{@{}c@{}}%
9\\-3\\1
\end{tabular}\endgroup%
\kern3pt%
\begingroup \smaller\smaller\smaller\begin{tabular}{@{}c@{}}%
5\\-1\\1
\end{tabular}\endgroup%
\kern3pt%
\begingroup \smaller\smaller\smaller\begin{tabular}{@{}c@{}}%
90\\-3\\19
\end{tabular}\endgroup%
\kern3pt%
\begingroup \smaller\smaller\smaller\begin{tabular}{@{}c@{}}%
90\\3\\19
\end{tabular}\endgroup%
\kern3pt%
\begingroup \smaller\smaller\smaller\begin{tabular}{@{}c@{}}%
30\\4\\6
\end{tabular}\endgroup%
\kern3pt%
\begingroup \smaller\smaller\smaller\begin{tabular}{@{}c@{}}%
30\\7\\5
\end{tabular}\endgroup%
\kern3pt%
\begingroup \smaller\smaller\smaller\begin{tabular}{@{}c@{}}%
9\\3\\1
\end{tabular}\endgroup%
\kern3pt%
\begingroup \smaller\smaller\smaller\begin{tabular}{@{}c@{}}%
30\\11\\1
\end{tabular}\endgroup%
{$\left.\llap{\phantom{%
\begingroup \smaller\smaller\smaller\begin{tabular}{@{}c@{}}%
0\\0\\0
\end{tabular}\endgroup%
}}\!\right]$}%
}%
\ifdim\wd\matricesbox>\halfwidth\myboxwidth=\hsize\else\myboxwidth=\halfwidth\fi
\vbox{%
\ifdim\myboxwidth=\hsize
\setbox\onelinebox=\hbox{%
\vbox{\hbox{%
$\Pi_{17,4}$ spans $L_{16.13}$%
}\hbox{%
$22|22236322\slashthree223632\rtimes D_{2}$%
}%
}%
\hfill\copy\matricesbox
}%
\ifdim\wd\onelinebox>\myboxwidth
\hbox to \myboxwidth{%
$\Pi_{17,4}$ spans $L_{16.13}$%
\hfil
$22|22236322\slashthree223632\rtimes D_{2}$%
}%
\box\matricesbox
\else
\hbox to \myboxwidth{%
\unhbox\onelinebox
}%
\fi
\else
\hbox to \myboxwidth{%
$\Pi_{17,4}$ spans $L_{16.13}$%
\hfil}%
\hbox to \myboxwidth{%
$22|22236322\slashthree223632\rtimes D_{2}$%
\hfil}%
\box\matricesbox
\fi
}%
\hfill\discretionary{}{}{}%
\setbox\matricesbox=\hbox{%
{$\left[\!\llap{\phantom{%
\begingroup \smaller\smaller\smaller\begin{tabular}{@{}c@{}}%
\phantom{0}\\\phantom{0}\\\phantom{0}
\end{tabular}\endgroup%
}}\right.$}%
\begingroup \smaller\smaller\smaller\begin{tabular}{@{}c@{}}%
-1\\\phantom{0}\\\phantom{0}
\end{tabular}\endgroup%
\kern3pt%
\begingroup \smaller\smaller\smaller\begin{tabular}{@{}c@{}}%
\phantom{0}\\5\\\phantom{0}
\end{tabular}\endgroup%
\kern3pt%
\begingroup \smaller\smaller\smaller\begin{tabular}{@{}c@{}}%
\phantom{0}\\\phantom{0}\\15
\end{tabular}\endgroup%
{$\left.\llap{\phantom{%
\begingroup \smaller\smaller\smaller\begin{tabular}{@{}c@{}}%
\phantom{0}\\\phantom{0}\\\phantom{0}
\end{tabular}\endgroup%
}}\!\right]$}%
{$\left[\!\llap{\phantom{%
\begingroup \smaller\smaller\smaller\begin{tabular}{@{}c@{}}%
0\\0\\0
\end{tabular}\endgroup%
}}\right.$}%
\begingroup \smaller\smaller\smaller\begin{tabular}{@{}c@{}}%
4\\-2\\0
\end{tabular}\endgroup%
\kern3pt%
\begingroup \smaller\smaller\smaller\begin{tabular}{@{}c@{}}%
15\\-6\\-2
\end{tabular}\endgroup%
\kern3pt%
\begingroup \smaller\smaller\smaller\begin{tabular}{@{}c@{}}%
20\\-6\\-4
\end{tabular}\endgroup%
\kern3pt%
\begingroup \smaller\smaller\smaller\begin{tabular}{@{}c@{}}%
20\\-3\\-5
\end{tabular}\endgroup%
\kern3pt%
\begingroup \smaller\smaller\smaller\begin{tabular}{@{}c@{}}%
15\\0\\-4
\end{tabular}\endgroup%
\kern3pt%
\begingroup \smaller\smaller\smaller\begin{tabular}{@{}c@{}}%
20\\3\\-5
\end{tabular}\endgroup%
\kern3pt%
\begingroup \smaller\smaller\smaller\begin{tabular}{@{}c@{}}%
20\\6\\-4
\end{tabular}\endgroup%
\kern3pt%
\begingroup \smaller\smaller\smaller\begin{tabular}{@{}c@{}}%
15\\6\\-2
\end{tabular}\endgroup%
\kern3pt%
\begingroup \smaller\smaller\smaller\begin{tabular}{@{}c@{}}%
20\\9\\-1
\end{tabular}\endgroup%
{$\left.\llap{\phantom{%
\begingroup \smaller\smaller\smaller\begin{tabular}{@{}c@{}}%
0\\0\\0
\end{tabular}\endgroup%
}}\!\right]$}%
}%
\ifdim\wd\matricesbox>\halfwidth\myboxwidth=\hsize\else\myboxwidth=\halfwidth\fi
\vbox{%
\ifdim\myboxwidth=\hsize
\setbox\onelinebox=\hbox{%
\vbox{\hbox{%
$\Pi_{17,5}$ spans $L_{31.7}$%
}\hbox{%
$|22322322\slashthree22322322\rtimes D_{2}$%
}%
}%
\hfill\copy\matricesbox
}%
\ifdim\wd\onelinebox>\myboxwidth
\hbox to \myboxwidth{%
$\Pi_{17,5}$ spans $L_{31.7}$%
\hfil
$|22322322\slashthree22322322\rtimes D_{2}$%
}%
\box\matricesbox
\else
\hbox to \myboxwidth{%
\unhbox\onelinebox
}%
\fi
\else
\hbox to \myboxwidth{%
$\Pi_{17,5}$ spans $L_{31.7}$%
\hfil}%
\hbox to \myboxwidth{%
$|22322322\slashthree22322322\rtimes D_{2}$%
\hfil}%
\box\matricesbox
\fi
}%
\hfill\discretionary{}{}{}%
\setbox\matricesbox=\hbox{%
{$\left[\!\llap{\phantom{%
\begingroup \smaller\smaller\smaller\begin{tabular}{@{}c@{}}%
\phantom{0}\\\phantom{0}\\\phantom{0}
\end{tabular}\endgroup%
}}\right.$}%
\begingroup \smaller\smaller\smaller\begin{tabular}{@{}c@{}}%
-1\\\phantom{0}\\\phantom{0}
\end{tabular}\endgroup%
\kern3pt%
\begingroup \smaller\smaller\smaller\begin{tabular}{@{}c@{}}%
\phantom{0}\\45/2\\\phantom{0}
\end{tabular}\endgroup%
\kern3pt%
\begingroup \smaller\smaller\smaller\begin{tabular}{@{}c@{}}%
\phantom{0}\\\phantom{0}\\15/2
\end{tabular}\endgroup%
{$\left.\llap{\phantom{%
\begingroup \smaller\smaller\smaller\begin{tabular}{@{}c@{}}%
\phantom{0}\\\phantom{0}\\\phantom{0}
\end{tabular}\endgroup%
}}\!\right]$}%
{$\left[\!\llap{\phantom{%
\begingroup \smaller\smaller\smaller\begin{tabular}{@{}c@{}}%
0\\0\\0
\end{tabular}\endgroup%
}}\right.$}%
\begingroup \smaller\smaller\smaller\begin{tabular}{@{}c@{}}%
9\\2\\0
\end{tabular}\endgroup%
\kern3pt%
\begingroup \smaller\smaller\smaller\begin{tabular}{@{}c@{}}%
5\\1\\1
\end{tabular}\endgroup%
\kern3pt%
\begingroup \smaller\smaller\smaller\begin{tabular}{@{}c@{}}%
9\\1\\3
\end{tabular}\endgroup%
\kern3pt%
\begingroup \smaller\smaller\smaller\begin{tabular}{@{}c@{}}%
30\\1\\11
\end{tabular}\endgroup%
\kern3pt%
\begingroup \smaller\smaller\smaller\begin{tabular}{@{}c@{}}%
30\\-1\\11
\end{tabular}\endgroup%
\kern3pt%
\begingroup \smaller\smaller\smaller\begin{tabular}{@{}c@{}}%
90\\-8\\30
\end{tabular}\endgroup%
\kern3pt%
\begingroup \smaller\smaller\smaller\begin{tabular}{@{}c@{}}%
90\\-11\\27
\end{tabular}\endgroup%
\kern3pt%
\begingroup \smaller\smaller\smaller\begin{tabular}{@{}c@{}}%
5\\-1\\1
\end{tabular}\endgroup%
\kern3pt%
\begingroup \smaller\smaller\smaller\begin{tabular}{@{}c@{}}%
90\\-19\\3
\end{tabular}\endgroup%
{$\left.\llap{\phantom{%
\begingroup \smaller\smaller\smaller\begin{tabular}{@{}c@{}}%
0\\0\\0
\end{tabular}\endgroup%
}}\!\right]$}%
}%
\ifdim\wd\matricesbox>\halfwidth\myboxwidth=\hsize\else\myboxwidth=\halfwidth\fi
\vbox{%
\ifdim\myboxwidth=\hsize
\setbox\onelinebox=\hbox{%
\vbox{\hbox{%
$\Pi_{17,6}$ spans $L_{16.13}$%
}\hbox{%
$363222|22236322\slashthree22\rtimes D_{2}$%
}%
}%
\hfill\copy\matricesbox
}%
\ifdim\wd\onelinebox>\myboxwidth
\hbox to \myboxwidth{%
$\Pi_{17,6}$ spans $L_{16.13}$%
\hfil
$363222|22236322\slashthree22\rtimes D_{2}$%
}%
\box\matricesbox
\else
\hbox to \myboxwidth{%
\unhbox\onelinebox
}%
\fi
\else
\hbox to \myboxwidth{%
$\Pi_{17,6}$ spans $L_{16.13}$%
\hfil}%
\hbox to \myboxwidth{%
$363222|22236322\slashthree22\rtimes D_{2}$%
\hfil}%
\box\matricesbox
\fi
}%
\hfill\discretionary{}{}{}%
\setbox\matricesbox=\hbox{%
{$\left[\!\llap{\phantom{%
\begingroup \smaller\smaller\smaller\begin{tabular}{@{}c@{}}%
\phantom{0}\\\phantom{0}\\\phantom{0}
\end{tabular}\endgroup%
}}\right.$}%
\begingroup \smaller\smaller\smaller\begin{tabular}{@{}c@{}}%
-1\\\phantom{0}\\\phantom{0}
\end{tabular}\endgroup%
\kern3pt%
\begingroup \smaller\smaller\smaller\begin{tabular}{@{}c@{}}%
\phantom{0}\\15/2\\\phantom{0}
\end{tabular}\endgroup%
\kern3pt%
\begingroup \smaller\smaller\smaller\begin{tabular}{@{}c@{}}%
\phantom{0}\\\phantom{0}\\45/2
\end{tabular}\endgroup%
{$\left.\llap{\phantom{%
\begingroup \smaller\smaller\smaller\begin{tabular}{@{}c@{}}%
\phantom{0}\\\phantom{0}\\\phantom{0}
\end{tabular}\endgroup%
}}\!\right]$}%
{$\left[\!\llap{\phantom{%
\begingroup \smaller\smaller\smaller\begin{tabular}{@{}c@{}}%
0\\0\\0
\end{tabular}\endgroup%
}}\right.$}%
\begingroup \smaller\smaller\smaller\begin{tabular}{@{}c@{}}%
30\\11\\1
\end{tabular}\endgroup%
\kern3pt%
\begingroup \smaller\smaller\smaller\begin{tabular}{@{}c@{}}%
9\\3\\1
\end{tabular}\endgroup%
\kern3pt%
\begingroup \smaller\smaller\smaller\begin{tabular}{@{}c@{}}%
5\\1\\1
\end{tabular}\endgroup%
\kern3pt%
\begingroup \smaller\smaller\smaller\begin{tabular}{@{}c@{}}%
9\\0\\2
\end{tabular}\endgroup%
\kern3pt%
\begingroup \smaller\smaller\smaller\begin{tabular}{@{}c@{}}%
30\\-4\\6
\end{tabular}\endgroup%
\kern3pt%
\begingroup \smaller\smaller\smaller\begin{tabular}{@{}c@{}}%
30\\-7\\5
\end{tabular}\endgroup%
\kern3pt%
\begingroup \smaller\smaller\smaller\begin{tabular}{@{}c@{}}%
90\\-27\\11
\end{tabular}\endgroup%
\kern3pt%
\begingroup \smaller\smaller\smaller\begin{tabular}{@{}c@{}}%
90\\-30\\8
\end{tabular}\endgroup%
\kern3pt%
\begingroup \smaller\smaller\smaller\begin{tabular}{@{}c@{}}%
5\\-2\\0
\end{tabular}\endgroup%
{$\left.\llap{\phantom{%
\begingroup \smaller\smaller\smaller\begin{tabular}{@{}c@{}}%
0\\0\\0
\end{tabular}\endgroup%
}}\!\right]$}%
}%
\ifdim\wd\matricesbox>\halfwidth\myboxwidth=\hsize\else\myboxwidth=\halfwidth\fi
\vbox{%
\ifdim\myboxwidth=\hsize
\setbox\onelinebox=\hbox{%
\vbox{\hbox{%
$\Pi_{17,7}$ spans $L_{16.13}$%
}\hbox{%
$3632222\slashthree22223632|2\rtimes D_{2}$%
}%
}%
\hfill\copy\matricesbox
}%
\ifdim\wd\onelinebox>\myboxwidth
\hbox to \myboxwidth{%
$\Pi_{17,7}$ spans $L_{16.13}$%
\hfil
$3632222\slashthree22223632|2\rtimes D_{2}$%
}%
\box\matricesbox
\else
\hbox to \myboxwidth{%
\unhbox\onelinebox
}%
\fi
\else
\hbox to \myboxwidth{%
$\Pi_{17,7}$ spans $L_{16.13}$%
\hfil}%
\hbox to \myboxwidth{%
$3632222\slashthree22223632|2\rtimes D_{2}$%
\hfil}%
\box\matricesbox
\fi
}%
\hfill\discretionary{}{}{}%
\setbox\matricesbox=\hbox{%
{$\left[\!\llap{\phantom{%
\begingroup \smaller\smaller\smaller\begin{tabular}{@{}c@{}}%
\phantom{0}\\\phantom{0}\\\phantom{0}
\end{tabular}\endgroup%
}}\right.$}%
\begingroup \smaller\smaller\smaller\begin{tabular}{@{}c@{}}%
-1\\\phantom{0}\\\phantom{0}
\end{tabular}\endgroup%
\kern3pt%
\begingroup \smaller\smaller\smaller\begin{tabular}{@{}c@{}}%
\phantom{0}\\45/2\\\phantom{0}
\end{tabular}\endgroup%
\kern3pt%
\begingroup \smaller\smaller\smaller\begin{tabular}{@{}c@{}}%
\phantom{0}\\\phantom{0}\\15/2
\end{tabular}\endgroup%
{$\left.\llap{\phantom{%
\begingroup \smaller\smaller\smaller\begin{tabular}{@{}c@{}}%
\phantom{0}\\\phantom{0}\\\phantom{0}
\end{tabular}\endgroup%
}}\!\right]$}%
{$\left[\!\llap{\phantom{%
\begingroup \smaller\smaller\smaller\begin{tabular}{@{}c@{}}%
0\\0\\0
\end{tabular}\endgroup%
}}\right.$}%
\begingroup \smaller\smaller\smaller\begin{tabular}{@{}c@{}}%
90\\19\\3
\end{tabular}\endgroup%
\kern3pt%
\begingroup \smaller\smaller\smaller\begin{tabular}{@{}c@{}}%
30\\6\\4
\end{tabular}\endgroup%
\kern3pt%
\begingroup \smaller\smaller\smaller\begin{tabular}{@{}c@{}}%
30\\5\\7
\end{tabular}\endgroup%
\kern3pt%
\begingroup \smaller\smaller\smaller\begin{tabular}{@{}c@{}}%
9\\1\\3
\end{tabular}\endgroup%
\kern3pt%
\begingroup \smaller\smaller\smaller\begin{tabular}{@{}c@{}}%
5\\0\\2
\end{tabular}\endgroup%
\kern3pt%
\begingroup \smaller\smaller\smaller\begin{tabular}{@{}c@{}}%
9\\-1\\3
\end{tabular}\endgroup%
\kern3pt%
\begingroup \smaller\smaller\smaller\begin{tabular}{@{}c@{}}%
30\\-5\\7
\end{tabular}\endgroup%
\kern3pt%
\begingroup \smaller\smaller\smaller\begin{tabular}{@{}c@{}}%
30\\-6\\4
\end{tabular}\endgroup%
\kern3pt%
\begingroup \smaller\smaller\smaller\begin{tabular}{@{}c@{}}%
9\\-2\\0
\end{tabular}\endgroup%
{$\left.\llap{\phantom{%
\begingroup \smaller\smaller\smaller\begin{tabular}{@{}c@{}}%
0\\0\\0
\end{tabular}\endgroup%
}}\!\right]$}%
}%
\ifdim\wd\matricesbox>\halfwidth\myboxwidth=\hsize\else\myboxwidth=\halfwidth\fi
\vbox{%
\ifdim\myboxwidth=\hsize
\setbox\onelinebox=\hbox{%
\vbox{\hbox{%
$\Pi_{17,8}$ spans $L_{16.13}$%
}\hbox{%
$\slashthree63222232|23222236\rtimes D_{2}$%
}%
}%
\hfill\copy\matricesbox
}%
\ifdim\wd\onelinebox>\myboxwidth
\hbox to \myboxwidth{%
$\Pi_{17,8}$ spans $L_{16.13}$%
\hfil
$\slashthree63222232|23222236\rtimes D_{2}$%
}%
\box\matricesbox
\else
\hbox to \myboxwidth{%
\unhbox\onelinebox
}%
\fi
\else
\hbox to \myboxwidth{%
$\Pi_{17,8}$ spans $L_{16.13}$%
\hfil}%
\hbox to \myboxwidth{%
$\slashthree63222232|23222236\rtimes D_{2}$%
\hfil}%
\box\matricesbox
\fi
}%
\hfill\discretionary{}{}{}%
\setbox\matricesbox=\hbox{%
{$\left[\!\llap{\phantom{%
\begingroup \smaller\smaller\smaller\begin{tabular}{@{}c@{}}%
\phantom{0}\\\phantom{0}\\\phantom{0}
\end{tabular}\endgroup%
}}\right.$}%
\begingroup \smaller\smaller\smaller\begin{tabular}{@{}c@{}}%
-1\\\phantom{0}\\\phantom{0}
\end{tabular}\endgroup%
\kern3pt%
\begingroup \smaller\smaller\smaller\begin{tabular}{@{}c@{}}%
\phantom{0}\\15/2\\\phantom{0}
\end{tabular}\endgroup%
\kern3pt%
\begingroup \smaller\smaller\smaller\begin{tabular}{@{}c@{}}%
\phantom{0}\\\phantom{0}\\45/2
\end{tabular}\endgroup%
{$\left.\llap{\phantom{%
\begingroup \smaller\smaller\smaller\begin{tabular}{@{}c@{}}%
\phantom{0}\\\phantom{0}\\\phantom{0}
\end{tabular}\endgroup%
}}\!\right]$}%
{$\left[\!\llap{\phantom{%
\begingroup \smaller\smaller\smaller\begin{tabular}{@{}c@{}}%
0\\0\\0
\end{tabular}\endgroup%
}}\right.$}%
\begingroup \smaller\smaller\smaller\begin{tabular}{@{}c@{}}%
5\\-2\\0
\end{tabular}\endgroup%
\kern3pt%
\begingroup \smaller\smaller\smaller\begin{tabular}{@{}c@{}}%
9\\-3\\-1
\end{tabular}\endgroup%
\kern3pt%
\begingroup \smaller\smaller\smaller\begin{tabular}{@{}c@{}}%
30\\-7\\-5
\end{tabular}\endgroup%
\kern3pt%
\begingroup \smaller\smaller\smaller\begin{tabular}{@{}c@{}}%
30\\-4\\-6
\end{tabular}\endgroup%
\kern3pt%
\begingroup \smaller\smaller\smaller\begin{tabular}{@{}c@{}}%
90\\-3\\-19
\end{tabular}\endgroup%
\kern3pt%
\begingroup \smaller\smaller\smaller\begin{tabular}{@{}c@{}}%
90\\3\\-19
\end{tabular}\endgroup%
\kern3pt%
\begingroup \smaller\smaller\smaller\begin{tabular}{@{}c@{}}%
5\\1\\-1
\end{tabular}\endgroup%
\kern3pt%
\begingroup \smaller\smaller\smaller\begin{tabular}{@{}c@{}}%
9\\3\\-1
\end{tabular}\endgroup%
\kern3pt%
\begingroup \smaller\smaller\smaller\begin{tabular}{@{}c@{}}%
30\\11\\-1
\end{tabular}\endgroup%
{$\left.\llap{\phantom{%
\begingroup \smaller\smaller\smaller\begin{tabular}{@{}c@{}}%
0\\0\\0
\end{tabular}\endgroup%
}}\!\right]$}%
}%
\ifdim\wd\matricesbox>\halfwidth\myboxwidth=\hsize\else\myboxwidth=\halfwidth\fi
\vbox{%
\ifdim\myboxwidth=\hsize
\setbox\onelinebox=\hbox{%
\vbox{\hbox{%
$\Pi_{17,9}$ spans $L_{16.13}$%
}\hbox{%
$36322|22363222\slashthree222\rtimes D_{2}$%
}%
}%
\hfill\copy\matricesbox
}%
\ifdim\wd\onelinebox>\myboxwidth
\hbox to \myboxwidth{%
$\Pi_{17,9}$ spans $L_{16.13}$%
\hfil
$36322|22363222\slashthree222\rtimes D_{2}$%
}%
\box\matricesbox
\else
\hbox to \myboxwidth{%
\unhbox\onelinebox
}%
\fi
\else
\hbox to \myboxwidth{%
$\Pi_{17,9}$ spans $L_{16.13}$%
\hfil}%
\hbox to \myboxwidth{%
$36322|22363222\slashthree222\rtimes D_{2}$%
\hfil}%
\box\matricesbox
\fi
}%
\hfill\discretionary{}{}{}%
\setbox\matricesbox=\hbox{%
{$\left[\!\llap{\phantom{%
\begingroup \smaller\smaller\smaller\begin{tabular}{@{}c@{}}%
\phantom{0}\\\phantom{0}\\\phantom{0}
\end{tabular}\endgroup%
}}\right.$}%
\begingroup \smaller\smaller\smaller\begin{tabular}{@{}c@{}}%
-1\\\phantom{0}\\\phantom{0}
\end{tabular}\endgroup%
\kern3pt%
\begingroup \smaller\smaller\smaller\begin{tabular}{@{}c@{}}%
\phantom{0}\\45/2\\\phantom{0}
\end{tabular}\endgroup%
\kern3pt%
\begingroup \smaller\smaller\smaller\begin{tabular}{@{}c@{}}%
\phantom{0}\\\phantom{0}\\15/2
\end{tabular}\endgroup%
{$\left.\llap{\phantom{%
\begingroup \smaller\smaller\smaller\begin{tabular}{@{}c@{}}%
\phantom{0}\\\phantom{0}\\\phantom{0}
\end{tabular}\endgroup%
}}\!\right]$}%
{$\left[\!\llap{\phantom{%
\begingroup \smaller\smaller\smaller\begin{tabular}{@{}c@{}}%
0\\0\\0
\end{tabular}\endgroup%
}}\right.$}%
\begingroup \smaller\smaller\smaller\begin{tabular}{@{}c@{}}%
90\\19\\3
\end{tabular}\endgroup%
\kern3pt%
\begingroup \smaller\smaller\smaller\begin{tabular}{@{}c@{}}%
30\\6\\4
\end{tabular}\endgroup%
\kern3pt%
\begingroup \smaller\smaller\smaller\begin{tabular}{@{}c@{}}%
30\\5\\7
\end{tabular}\endgroup%
\kern3pt%
\begingroup \smaller\smaller\smaller\begin{tabular}{@{}c@{}}%
9\\1\\3
\end{tabular}\endgroup%
\kern3pt%
\begingroup \smaller\smaller\smaller\begin{tabular}{@{}c@{}}%
5\\0\\2
\end{tabular}\endgroup%
\kern3pt%
\begingroup \smaller\smaller\smaller\begin{tabular}{@{}c@{}}%
90\\-8\\30
\end{tabular}\endgroup%
\kern3pt%
\begingroup \smaller\smaller\smaller\begin{tabular}{@{}c@{}}%
90\\-11\\27
\end{tabular}\endgroup%
\kern3pt%
\begingroup \smaller\smaller\smaller\begin{tabular}{@{}c@{}}%
5\\-1\\1
\end{tabular}\endgroup%
\kern3pt%
\begingroup \smaller\smaller\smaller\begin{tabular}{@{}c@{}}%
9\\-2\\0
\end{tabular}\endgroup%
{$\left.\llap{\phantom{%
\begingroup \smaller\smaller\smaller\begin{tabular}{@{}c@{}}%
0\\0\\0
\end{tabular}\endgroup%
}}\!\right]$}%
}%
\ifdim\wd\matricesbox>\halfwidth\myboxwidth=\hsize\else\myboxwidth=\halfwidth\fi
\vbox{%
\ifdim\myboxwidth=\hsize
\setbox\onelinebox=\hbox{%
\vbox{\hbox{%
$\Pi_{17,10}$ spans $L_{16.13}$%
}\hbox{%
$\slashthree63222322|22322236\rtimes D_{2}$%
}%
}%
\hfill\copy\matricesbox
}%
\ifdim\wd\onelinebox>\myboxwidth
\hbox to \myboxwidth{%
$\Pi_{17,10}$ spans $L_{16.13}$%
\hfil
$\slashthree63222322|22322236\rtimes D_{2}$%
}%
\box\matricesbox
\else
\hbox to \myboxwidth{%
\unhbox\onelinebox
}%
\fi
\else
\hbox to \myboxwidth{%
$\Pi_{17,10}$ spans $L_{16.13}$%
\hfil}%
\hbox to \myboxwidth{%
$\slashthree63222322|22322236\rtimes D_{2}$%
\hfil}%
\box\matricesbox
\fi
}%
\hfill\discretionary{}{}{}%
\setbox\matricesbox=\hbox{%
{$\left[\!\llap{\phantom{%
\begingroup \smaller\smaller\smaller\begin{tabular}{@{}c@{}}%
\phantom{0}\\\phantom{0}\\\phantom{0}
\end{tabular}\endgroup%
}}\right.$}%
\begingroup \smaller\smaller\smaller\begin{tabular}{@{}c@{}}%
-1\\\phantom{0}\\\phantom{0}
\end{tabular}\endgroup%
\kern3pt%
\begingroup \smaller\smaller\smaller\begin{tabular}{@{}c@{}}%
\phantom{0}\\45/2\\\phantom{0}
\end{tabular}\endgroup%
\kern3pt%
\begingroup \smaller\smaller\smaller\begin{tabular}{@{}c@{}}%
\phantom{0}\\\phantom{0}\\15/2
\end{tabular}\endgroup%
{$\left.\llap{\phantom{%
\begingroup \smaller\smaller\smaller\begin{tabular}{@{}c@{}}%
\phantom{0}\\\phantom{0}\\\phantom{0}
\end{tabular}\endgroup%
}}\!\right]$}%
{$\left[\!\llap{\phantom{%
\begingroup \smaller\smaller\smaller\begin{tabular}{@{}c@{}}%
0\\0\\0
\end{tabular}\endgroup%
}}\right.$}%
\begingroup \smaller\smaller\smaller\begin{tabular}{@{}c@{}}%
90\\19\\3
\end{tabular}\endgroup%
\kern3pt%
\begingroup \smaller\smaller\smaller\begin{tabular}{@{}c@{}}%
5\\1\\1
\end{tabular}\endgroup%
\kern3pt%
\begingroup \smaller\smaller\smaller\begin{tabular}{@{}c@{}}%
9\\1\\3
\end{tabular}\endgroup%
\kern3pt%
\begingroup \smaller\smaller\smaller\begin{tabular}{@{}c@{}}%
30\\1\\11
\end{tabular}\endgroup%
\kern3pt%
\begingroup \smaller\smaller\smaller\begin{tabular}{@{}c@{}}%
30\\-1\\11
\end{tabular}\endgroup%
\kern3pt%
\begingroup \smaller\smaller\smaller\begin{tabular}{@{}c@{}}%
90\\-8\\30
\end{tabular}\endgroup%
\kern3pt%
\begingroup \smaller\smaller\smaller\begin{tabular}{@{}c@{}}%
90\\-11\\27
\end{tabular}\endgroup%
\kern3pt%
\begingroup \smaller\smaller\smaller\begin{tabular}{@{}c@{}}%
5\\-1\\1
\end{tabular}\endgroup%
\kern3pt%
\begingroup \smaller\smaller\smaller\begin{tabular}{@{}c@{}}%
9\\-2\\0
\end{tabular}\endgroup%
{$\left.\llap{\phantom{%
\begingroup \smaller\smaller\smaller\begin{tabular}{@{}c@{}}%
0\\0\\0
\end{tabular}\endgroup%
}}\!\right]$}%
}%
\ifdim\wd\matricesbox>\halfwidth\myboxwidth=\hsize\else\myboxwidth=\halfwidth\fi
\vbox{%
\ifdim\myboxwidth=\hsize
\setbox\onelinebox=\hbox{%
\vbox{\hbox{%
$\Pi_{17,11}$ spans $L_{16.13}$%
}\hbox{%
$363222\slashthree22236322|22\rtimes D_{2}$%
}%
}%
\hfill\copy\matricesbox
}%
\ifdim\wd\onelinebox>\myboxwidth
\hbox to \myboxwidth{%
$\Pi_{17,11}$ spans $L_{16.13}$%
\hfil
$363222\slashthree22236322|22\rtimes D_{2}$%
}%
\box\matricesbox
\else
\hbox to \myboxwidth{%
\unhbox\onelinebox
}%
\fi
\else
\hbox to \myboxwidth{%
$\Pi_{17,11}$ spans $L_{16.13}$%
\hfil}%
\hbox to \myboxwidth{%
$363222\slashthree22236322|22\rtimes D_{2}$%
\hfil}%
\box\matricesbox
\fi
}%
\hfill\discretionary{}{}{}%
\setbox\matricesbox=\hbox{%
{$\left[\!\llap{\phantom{%
\begingroup \smaller\smaller\smaller\begin{tabular}{@{}c@{}}%
\phantom{0}\\\phantom{0}\\\phantom{0}
\end{tabular}\endgroup%
}}\right.$}%
\begingroup \smaller\smaller\smaller\begin{tabular}{@{}c@{}}%
-1\\\phantom{0}\\\phantom{0}
\end{tabular}\endgroup%
\kern3pt%
\begingroup \smaller\smaller\smaller\begin{tabular}{@{}c@{}}%
\phantom{0}\\45/2\\\phantom{0}
\end{tabular}\endgroup%
\kern3pt%
\begingroup \smaller\smaller\smaller\begin{tabular}{@{}c@{}}%
\phantom{0}\\\phantom{0}\\15/2
\end{tabular}\endgroup%
{$\left.\llap{\phantom{%
\begingroup \smaller\smaller\smaller\begin{tabular}{@{}c@{}}%
\phantom{0}\\\phantom{0}\\\phantom{0}
\end{tabular}\endgroup%
}}\!\right]$}%
{$\left[\!\llap{\phantom{%
\begingroup \smaller\smaller\smaller\begin{tabular}{@{}c@{}}%
0\\0\\0
\end{tabular}\endgroup%
}}\right.$}%
\begingroup \smaller\smaller\smaller\begin{tabular}{@{}c@{}}%
90\\19\\3
\end{tabular}\endgroup%
\kern3pt%
\begingroup \smaller\smaller\smaller\begin{tabular}{@{}c@{}}%
30\\6\\4
\end{tabular}\endgroup%
\kern3pt%
\begingroup \smaller\smaller\smaller\begin{tabular}{@{}c@{}}%
30\\5\\7
\end{tabular}\endgroup%
\kern3pt%
\begingroup \smaller\smaller\smaller\begin{tabular}{@{}c@{}}%
9\\1\\3
\end{tabular}\endgroup%
\kern3pt%
\begingroup \smaller\smaller\smaller\begin{tabular}{@{}c@{}}%
30\\1\\11
\end{tabular}\endgroup%
\kern3pt%
\begingroup \smaller\smaller\smaller\begin{tabular}{@{}c@{}}%
30\\-1\\11
\end{tabular}\endgroup%
\kern3pt%
\begingroup \smaller\smaller\smaller\begin{tabular}{@{}c@{}}%
9\\-1\\3
\end{tabular}\endgroup%
\kern3pt%
\begingroup \smaller\smaller\smaller\begin{tabular}{@{}c@{}}%
5\\-1\\1
\end{tabular}\endgroup%
\kern3pt%
\begingroup \smaller\smaller\smaller\begin{tabular}{@{}c@{}}%
9\\-2\\0
\end{tabular}\endgroup%
{$\left.\llap{\phantom{%
\begingroup \smaller\smaller\smaller\begin{tabular}{@{}c@{}}%
0\\0\\0
\end{tabular}\endgroup%
}}\!\right]$}%
}%
\ifdim\wd\matricesbox>\halfwidth\myboxwidth=\hsize\else\myboxwidth=\halfwidth\fi
\vbox{%
\ifdim\myboxwidth=\hsize
\setbox\onelinebox=\hbox{%
\vbox{\hbox{%
$\Pi_{17,12}$ spans $L_{16.13}$%
}\hbox{%
$\slashthree63223222|22232236\rtimes D_{2}$%
}%
}%
\hfill\copy\matricesbox
}%
\ifdim\wd\onelinebox>\myboxwidth
\hbox to \myboxwidth{%
$\Pi_{17,12}$ spans $L_{16.13}$%
\hfil
$\slashthree63223222|22232236\rtimes D_{2}$%
}%
\box\matricesbox
\else
\hbox to \myboxwidth{%
\unhbox\onelinebox
}%
\fi
\else
\hbox to \myboxwidth{%
$\Pi_{17,12}$ spans $L_{16.13}$%
\hfil}%
\hbox to \myboxwidth{%
$\slashthree63223222|22232236\rtimes D_{2}$%
\hfil}%
\box\matricesbox
\fi
}%
\hfill\discretionary{}{}{}%
\setbox\matricesbox=\hbox{%
{$\left[\!\llap{\phantom{%
\begingroup \smaller\smaller\smaller\begin{tabular}{@{}c@{}}%
\phantom{0}\\\phantom{0}\\\phantom{0}
\end{tabular}\endgroup%
}}\right.$}%
\begingroup \smaller\smaller\smaller\begin{tabular}{@{}c@{}}%
-1\\\phantom{0}\\\phantom{0}
\end{tabular}\endgroup%
\kern3pt%
\begingroup \smaller\smaller\smaller\begin{tabular}{@{}c@{}}%
\phantom{0}\\45/2\\\phantom{0}
\end{tabular}\endgroup%
\kern3pt%
\begingroup \smaller\smaller\smaller\begin{tabular}{@{}c@{}}%
\phantom{0}\\\phantom{0}\\15/2
\end{tabular}\endgroup%
{$\left.\llap{\phantom{%
\begingroup \smaller\smaller\smaller\begin{tabular}{@{}c@{}}%
\phantom{0}\\\phantom{0}\\\phantom{0}
\end{tabular}\endgroup%
}}\!\right]$}%
{$\left[\!\llap{\phantom{%
\begingroup \smaller\smaller\smaller\begin{tabular}{@{}c@{}}%
0\\0\\0
\end{tabular}\endgroup%
}}\right.$}%
\begingroup \smaller\smaller\smaller\begin{tabular}{@{}c@{}}%
9\\-2\\0
\end{tabular}\endgroup%
\kern3pt%
\begingroup \smaller\smaller\smaller\begin{tabular}{@{}c@{}}%
30\\-6\\4
\end{tabular}\endgroup%
\kern3pt%
\begingroup \smaller\smaller\smaller\begin{tabular}{@{}c@{}}%
30\\-5\\7
\end{tabular}\endgroup%
\kern3pt%
\begingroup \smaller\smaller\smaller\begin{tabular}{@{}c@{}}%
90\\-11\\27
\end{tabular}\endgroup%
\kern3pt%
\begingroup \smaller\smaller\smaller\begin{tabular}{@{}c@{}}%
90\\-8\\30
\end{tabular}\endgroup%
\kern3pt%
\begingroup \smaller\smaller\smaller\begin{tabular}{@{}c@{}}%
5\\0\\2
\end{tabular}\endgroup%
\kern3pt%
\begingroup \smaller\smaller\smaller\begin{tabular}{@{}c@{}}%
9\\1\\3
\end{tabular}\endgroup%
\kern3pt%
\begingroup \smaller\smaller\smaller\begin{tabular}{@{}c@{}}%
5\\1\\1
\end{tabular}\endgroup%
\kern3pt%
\begingroup \smaller\smaller\smaller\begin{tabular}{@{}c@{}}%
90\\19\\3
\end{tabular}\endgroup%
{$\left.\llap{\phantom{%
\begingroup \smaller\smaller\smaller\begin{tabular}{@{}c@{}}%
0\\0\\0
\end{tabular}\endgroup%
}}\!\right]$}%
}%
\ifdim\wd\matricesbox>\halfwidth\myboxwidth=\hsize\else\myboxwidth=\halfwidth\fi
\vbox{%
\ifdim\myboxwidth=\hsize
\setbox\onelinebox=\hbox{%
\vbox{\hbox{%
$\Pi_{17,13}$ spans $L_{16.13}$%
}\hbox{%
$3632|23632222\slashthree2222\rtimes D_{2}$%
}%
}%
\hfill\copy\matricesbox
}%
\ifdim\wd\onelinebox>\myboxwidth
\hbox to \myboxwidth{%
$\Pi_{17,13}$ spans $L_{16.13}$%
\hfil
$3632|23632222\slashthree2222\rtimes D_{2}$%
}%
\box\matricesbox
\else
\hbox to \myboxwidth{%
\unhbox\onelinebox
}%
\fi
\else
\hbox to \myboxwidth{%
$\Pi_{17,13}$ spans $L_{16.13}$%
\hfil}%
\hbox to \myboxwidth{%
$3632|23632222\slashthree2222\rtimes D_{2}$%
\hfil}%
\box\matricesbox
\fi
}%
\hfill\discretionary{}{}{}%
\setbox\matricesbox=\hbox{%
{$\left[\!\llap{\phantom{%
\begingroup \smaller\smaller\smaller\begin{tabular}{@{}c@{}}%
\phantom{0}\\\phantom{0}\\\phantom{0}
\end{tabular}\endgroup%
}}\right.$}%
\begingroup \smaller\smaller\smaller\begin{tabular}{@{}c@{}}%
-1\\\phantom{0}\\\phantom{0}
\end{tabular}\endgroup%
\kern3pt%
\begingroup \smaller\smaller\smaller\begin{tabular}{@{}c@{}}%
\phantom{0}\\15/2\\\phantom{0}
\end{tabular}\endgroup%
\kern3pt%
\begingroup \smaller\smaller\smaller\begin{tabular}{@{}c@{}}%
\phantom{0}\\\phantom{0}\\45/2
\end{tabular}\endgroup%
{$\left.\llap{\phantom{%
\begingroup \smaller\smaller\smaller\begin{tabular}{@{}c@{}}%
\phantom{0}\\\phantom{0}\\\phantom{0}
\end{tabular}\endgroup%
}}\!\right]$}%
{$\left[\!\llap{\phantom{%
\begingroup \smaller\smaller\smaller\begin{tabular}{@{}c@{}}%
0\\0\\0
\end{tabular}\endgroup%
}}\right.$}%
\begingroup \smaller\smaller\smaller\begin{tabular}{@{}c@{}}%
30\\-11\\-1
\end{tabular}\endgroup%
\kern3pt%
\begingroup \smaller\smaller\smaller\begin{tabular}{@{}c@{}}%
90\\-30\\-8
\end{tabular}\endgroup%
\kern3pt%
\begingroup \smaller\smaller\smaller\begin{tabular}{@{}c@{}}%
90\\-27\\-11
\end{tabular}\endgroup%
\kern3pt%
\begingroup \smaller\smaller\smaller\begin{tabular}{@{}c@{}}%
5\\-1\\-1
\end{tabular}\endgroup%
\kern3pt%
\begingroup \smaller\smaller\smaller\begin{tabular}{@{}c@{}}%
9\\0\\-2
\end{tabular}\endgroup%
\kern3pt%
\begingroup \smaller\smaller\smaller\begin{tabular}{@{}c@{}}%
5\\1\\-1
\end{tabular}\endgroup%
\kern3pt%
\begingroup \smaller\smaller\smaller\begin{tabular}{@{}c@{}}%
90\\27\\-11
\end{tabular}\endgroup%
\kern3pt%
\begingroup \smaller\smaller\smaller\begin{tabular}{@{}c@{}}%
90\\30\\-8
\end{tabular}\endgroup%
\kern3pt%
\begingroup \smaller\smaller\smaller\begin{tabular}{@{}c@{}}%
5\\2\\0
\end{tabular}\endgroup%
{$\left.\llap{\phantom{%
\begingroup \smaller\smaller\smaller\begin{tabular}{@{}c@{}}%
0\\0\\0
\end{tabular}\endgroup%
}}\!\right]$}%
}%
\ifdim\wd\matricesbox>\halfwidth\myboxwidth=\hsize\else\myboxwidth=\halfwidth\fi
\vbox{%
\ifdim\myboxwidth=\hsize
\setbox\onelinebox=\hbox{%
\vbox{\hbox{%
$\Pi_{17,14}$ spans $L_{16.13}$%
}\hbox{%
$36\slashthree63222232|232222\rtimes D_{2}$%
}%
}%
\hfill\copy\matricesbox
}%
\ifdim\wd\onelinebox>\myboxwidth
\hbox to \myboxwidth{%
$\Pi_{17,14}$ spans $L_{16.13}$%
\hfil
$36\slashthree63222232|232222\rtimes D_{2}$%
}%
\box\matricesbox
\else
\hbox to \myboxwidth{%
\unhbox\onelinebox
}%
\fi
\else
\hbox to \myboxwidth{%
$\Pi_{17,14}$ spans $L_{16.13}$%
\hfil}%
\hbox to \myboxwidth{%
$36\slashthree63222232|232222\rtimes D_{2}$%
\hfil}%
\box\matricesbox
\fi
}%
\hfill\discretionary{}{}{}%
\setbox\matricesbox=\hbox{%
{$\left[\!\llap{\phantom{%
\begingroup \smaller\smaller\smaller\begin{tabular}{@{}c@{}}%
\phantom{0}\\\phantom{0}\\\phantom{0}
\end{tabular}\endgroup%
}}\right.$}%
\begingroup \smaller\smaller\smaller\begin{tabular}{@{}c@{}}%
-1\\\phantom{0}\\\phantom{0}
\end{tabular}\endgroup%
\kern3pt%
\begingroup \smaller\smaller\smaller\begin{tabular}{@{}c@{}}%
\phantom{0}\\15/2\\\phantom{0}
\end{tabular}\endgroup%
\kern3pt%
\begingroup \smaller\smaller\smaller\begin{tabular}{@{}c@{}}%
\phantom{0}\\\phantom{0}\\45/2
\end{tabular}\endgroup%
{$\left.\llap{\phantom{%
\begingroup \smaller\smaller\smaller\begin{tabular}{@{}c@{}}%
\phantom{0}\\\phantom{0}\\\phantom{0}
\end{tabular}\endgroup%
}}\!\right]$}%
{$\left[\!\llap{\phantom{%
\begingroup \smaller\smaller\smaller\begin{tabular}{@{}c@{}}%
0\\0\\0
\end{tabular}\endgroup%
}}\right.$}%
\begingroup \smaller\smaller\smaller\begin{tabular}{@{}c@{}}%
30\\-11\\-1
\end{tabular}\endgroup%
\kern3pt%
\begingroup \smaller\smaller\smaller\begin{tabular}{@{}c@{}}%
90\\-30\\-8
\end{tabular}\endgroup%
\kern3pt%
\begingroup \smaller\smaller\smaller\begin{tabular}{@{}c@{}}%
90\\-27\\-11
\end{tabular}\endgroup%
\kern3pt%
\begingroup \smaller\smaller\smaller\begin{tabular}{@{}c@{}}%
5\\-1\\-1
\end{tabular}\endgroup%
\kern3pt%
\begingroup \smaller\smaller\smaller\begin{tabular}{@{}c@{}}%
9\\0\\-2
\end{tabular}\endgroup%
\kern3pt%
\begingroup \smaller\smaller\smaller\begin{tabular}{@{}c@{}}%
30\\4\\-6
\end{tabular}\endgroup%
\kern3pt%
\begingroup \smaller\smaller\smaller\begin{tabular}{@{}c@{}}%
30\\7\\-5
\end{tabular}\endgroup%
\kern3pt%
\begingroup \smaller\smaller\smaller\begin{tabular}{@{}c@{}}%
9\\3\\-1
\end{tabular}\endgroup%
\kern3pt%
\begingroup \smaller\smaller\smaller\begin{tabular}{@{}c@{}}%
5\\2\\0
\end{tabular}\endgroup%
{$\left.\llap{\phantom{%
\begingroup \smaller\smaller\smaller\begin{tabular}{@{}c@{}}%
0\\0\\0
\end{tabular}\endgroup%
}}\!\right]$}%
}%
\ifdim\wd\matricesbox>\halfwidth\myboxwidth=\hsize\else\myboxwidth=\halfwidth\fi
\vbox{%
\ifdim\myboxwidth=\hsize
\setbox\onelinebox=\hbox{%
\vbox{\hbox{%
$\Pi_{17,15}$ spans $L_{16.13}$%
}\hbox{%
$36\slashthree63222322|223222\rtimes D_{2}$%
}%
}%
\hfill\copy\matricesbox
}%
\ifdim\wd\onelinebox>\myboxwidth
\hbox to \myboxwidth{%
$\Pi_{17,15}$ spans $L_{16.13}$%
\hfil
$36\slashthree63222322|223222\rtimes D_{2}$%
}%
\box\matricesbox
\else
\hbox to \myboxwidth{%
\unhbox\onelinebox
}%
\fi
\else
\hbox to \myboxwidth{%
$\Pi_{17,15}$ spans $L_{16.13}$%
\hfil}%
\hbox to \myboxwidth{%
$36\slashthree63222322|223222\rtimes D_{2}$%
\hfil}%
\box\matricesbox
\fi
}%
\hfill\discretionary{}{}{}%
\setbox\matricesbox=\hbox{%
{$\left[\!\llap{\phantom{%
\begingroup \smaller\smaller\smaller\begin{tabular}{@{}c@{}}%
\phantom{0}\\\phantom{0}\\\phantom{0}
\end{tabular}\endgroup%
}}\right.$}%
\begingroup \smaller\smaller\smaller\begin{tabular}{@{}c@{}}%
-1\\\phantom{0}\\\phantom{0}
\end{tabular}\endgroup%
\kern3pt%
\begingroup \smaller\smaller\smaller\begin{tabular}{@{}c@{}}%
\phantom{0}\\15/2\\\phantom{0}
\end{tabular}\endgroup%
\kern3pt%
\begingroup \smaller\smaller\smaller\begin{tabular}{@{}c@{}}%
\phantom{0}\\\phantom{0}\\45/2
\end{tabular}\endgroup%
{$\left.\llap{\phantom{%
\begingroup \smaller\smaller\smaller\begin{tabular}{@{}c@{}}%
\phantom{0}\\\phantom{0}\\\phantom{0}
\end{tabular}\endgroup%
}}\!\right]$}%
{$\left[\!\llap{\phantom{%
\begingroup \smaller\smaller\smaller\begin{tabular}{@{}c@{}}%
0\\0\\0
\end{tabular}\endgroup%
}}\right.$}%
\begingroup \smaller\smaller\smaller\begin{tabular}{@{}c@{}}%
5\\2\\0
\end{tabular}\endgroup%
\kern3pt%
\begingroup \smaller\smaller\smaller\begin{tabular}{@{}c@{}}%
9\\3\\1
\end{tabular}\endgroup%
\kern3pt%
\begingroup \smaller\smaller\smaller\begin{tabular}{@{}c@{}}%
30\\7\\5
\end{tabular}\endgroup%
\kern3pt%
\begingroup \smaller\smaller\smaller\begin{tabular}{@{}c@{}}%
30\\4\\6
\end{tabular}\endgroup%
\kern3pt%
\begingroup \smaller\smaller\smaller\begin{tabular}{@{}c@{}}%
9\\0\\2
\end{tabular}\endgroup%
\kern3pt%
\begingroup \smaller\smaller\smaller\begin{tabular}{@{}c@{}}%
30\\-4\\6
\end{tabular}\endgroup%
\kern3pt%
\begingroup \smaller\smaller\smaller\begin{tabular}{@{}c@{}}%
30\\-7\\5
\end{tabular}\endgroup%
\kern3pt%
\begingroup \smaller\smaller\smaller\begin{tabular}{@{}c@{}}%
9\\-3\\1
\end{tabular}\endgroup%
\kern3pt%
\begingroup \smaller\smaller\smaller\begin{tabular}{@{}c@{}}%
30\\-11\\1
\end{tabular}\endgroup%
{$\left.\llap{\phantom{%
\begingroup \smaller\smaller\smaller\begin{tabular}{@{}c@{}}%
0\\0\\0
\end{tabular}\endgroup%
}}\!\right]$}%
}%
\ifdim\wd\matricesbox>\halfwidth\myboxwidth=\hsize\else\myboxwidth=\halfwidth\fi
\vbox{%
\ifdim\myboxwidth=\hsize
\setbox\onelinebox=\hbox{%
\vbox{\hbox{%
$\Pi_{17,16}$ spans $L_{16.13}$%
}\hbox{%
$322|22322322\slashthree22322\rtimes D_{2}$%
}%
}%
\hfill\copy\matricesbox
}%
\ifdim\wd\onelinebox>\myboxwidth
\hbox to \myboxwidth{%
$\Pi_{17,16}$ spans $L_{16.13}$%
\hfil
$322|22322322\slashthree22322\rtimes D_{2}$%
}%
\box\matricesbox
\else
\hbox to \myboxwidth{%
\unhbox\onelinebox
}%
\fi
\else
\hbox to \myboxwidth{%
$\Pi_{17,16}$ spans $L_{16.13}$%
\hfil}%
\hbox to \myboxwidth{%
$322|22322322\slashthree22322\rtimes D_{2}$%
\hfil}%
\box\matricesbox
\fi
}%
\hfill\discretionary{}{}{}%
\setbox\matricesbox=\hbox{%
{$\left[\!\llap{\phantom{%
\begingroup \smaller\smaller\smaller\begin{tabular}{@{}c@{}}%
\phantom{0}\\\phantom{0}\\\phantom{0}
\end{tabular}\endgroup%
}}\right.$}%
\begingroup \smaller\smaller\smaller\begin{tabular}{@{}c@{}}%
-1\\\phantom{0}\\\phantom{0}
\end{tabular}\endgroup%
\kern3pt%
\begingroup \smaller\smaller\smaller\begin{tabular}{@{}c@{}}%
\phantom{0}\\45/2\\\phantom{0}
\end{tabular}\endgroup%
\kern3pt%
\begingroup \smaller\smaller\smaller\begin{tabular}{@{}c@{}}%
\phantom{0}\\\phantom{0}\\15/2
\end{tabular}\endgroup%
{$\left.\llap{\phantom{%
\begingroup \smaller\smaller\smaller\begin{tabular}{@{}c@{}}%
\phantom{0}\\\phantom{0}\\\phantom{0}
\end{tabular}\endgroup%
}}\!\right]$}%
{$\left[\!\llap{\phantom{%
\begingroup \smaller\smaller\smaller\begin{tabular}{@{}c@{}}%
0\\0\\0
\end{tabular}\endgroup%
}}\right.$}%
\begingroup \smaller\smaller\smaller\begin{tabular}{@{}c@{}}%
90\\19\\-3
\end{tabular}\endgroup%
\kern3pt%
\begingroup \smaller\smaller\smaller\begin{tabular}{@{}c@{}}%
5\\1\\-1
\end{tabular}\endgroup%
\kern3pt%
\begingroup \smaller\smaller\smaller\begin{tabular}{@{}c@{}}%
9\\1\\-3
\end{tabular}\endgroup%
\kern3pt%
\begingroup \smaller\smaller\smaller\begin{tabular}{@{}c@{}}%
30\\1\\-11
\end{tabular}\endgroup%
\kern3pt%
\begingroup \smaller\smaller\smaller\begin{tabular}{@{}c@{}}%
30\\-1\\-11
\end{tabular}\endgroup%
\kern3pt%
\begingroup \smaller\smaller\smaller\begin{tabular}{@{}c@{}}%
9\\-1\\-3
\end{tabular}\endgroup%
\kern3pt%
\begingroup \smaller\smaller\smaller\begin{tabular}{@{}c@{}}%
30\\-5\\-7
\end{tabular}\endgroup%
\kern3pt%
\begingroup \smaller\smaller\smaller\begin{tabular}{@{}c@{}}%
30\\-6\\-4
\end{tabular}\endgroup%
\kern3pt%
\begingroup \smaller\smaller\smaller\begin{tabular}{@{}c@{}}%
9\\-2\\0
\end{tabular}\endgroup%
{$\left.\llap{\phantom{%
\begingroup \smaller\smaller\smaller\begin{tabular}{@{}c@{}}%
0\\0\\0
\end{tabular}\endgroup%
}}\!\right]$}%
}%
\ifdim\wd\matricesbox>\halfwidth\myboxwidth=\hsize\else\myboxwidth=\halfwidth\fi
\vbox{%
\ifdim\myboxwidth=\hsize
\setbox\onelinebox=\hbox{%
\vbox{\hbox{%
$\Pi_{17,17}$ spans $L_{16.13}$%
}\hbox{%
$3222\slashthree22232232|2322\rtimes D_{2}$%
}%
}%
\hfill\copy\matricesbox
}%
\ifdim\wd\onelinebox>\myboxwidth
\hbox to \myboxwidth{%
$\Pi_{17,17}$ spans $L_{16.13}$%
\hfil
$3222\slashthree22232232|2322\rtimes D_{2}$%
}%
\box\matricesbox
\else
\hbox to \myboxwidth{%
\unhbox\onelinebox
}%
\fi
\else
\hbox to \myboxwidth{%
$\Pi_{17,17}$ spans $L_{16.13}$%
\hfil}%
\hbox to \myboxwidth{%
$3222\slashthree22232232|2322\rtimes D_{2}$%
\hfil}%
\box\matricesbox
\fi
}%
\hfill\discretionary{}{}{}%
\setbox\matricesbox=\hbox{%
{$\left[\!\llap{\phantom{%
\begingroup \smaller\smaller\smaller\begin{tabular}{@{}c@{}}%
\phantom{0}\\\phantom{0}\\\phantom{0}
\end{tabular}\endgroup%
}}\right.$}%
\begingroup \smaller\smaller\smaller\begin{tabular}{@{}c@{}}%
-1\\\phantom{0}\\\phantom{0}
\end{tabular}\endgroup%
\kern3pt%
\begingroup \smaller\smaller\smaller\begin{tabular}{@{}c@{}}%
\phantom{0}\\15/2\\\phantom{0}
\end{tabular}\endgroup%
\kern3pt%
\begingroup \smaller\smaller\smaller\begin{tabular}{@{}c@{}}%
\phantom{0}\\\phantom{0}\\45/2
\end{tabular}\endgroup%
{$\left.\llap{\phantom{%
\begingroup \smaller\smaller\smaller\begin{tabular}{@{}c@{}}%
\phantom{0}\\\phantom{0}\\\phantom{0}
\end{tabular}\endgroup%
}}\!\right]$}%
{$\left[\!\llap{\phantom{%
\begingroup \smaller\smaller\smaller\begin{tabular}{@{}c@{}}%
0\\0\\0
\end{tabular}\endgroup%
}}\right.$}%
\begingroup \smaller\smaller\smaller\begin{tabular}{@{}c@{}}%
5\\2\\0
\end{tabular}\endgroup%
\kern3pt%
\begingroup \smaller\smaller\smaller\begin{tabular}{@{}c@{}}%
90\\30\\8
\end{tabular}\endgroup%
\kern3pt%
\begingroup \smaller\smaller\smaller\begin{tabular}{@{}c@{}}%
90\\27\\11
\end{tabular}\endgroup%
\kern3pt%
\begingroup \smaller\smaller\smaller\begin{tabular}{@{}c@{}}%
5\\1\\1
\end{tabular}\endgroup%
\kern3pt%
\begingroup \smaller\smaller\smaller\begin{tabular}{@{}c@{}}%
9\\0\\2
\end{tabular}\endgroup%
\kern3pt%
\begingroup \smaller\smaller\smaller\begin{tabular}{@{}c@{}}%
30\\-4\\6
\end{tabular}\endgroup%
\kern3pt%
\begingroup \smaller\smaller\smaller\begin{tabular}{@{}c@{}}%
30\\-7\\5
\end{tabular}\endgroup%
\kern3pt%
\begingroup \smaller\smaller\smaller\begin{tabular}{@{}c@{}}%
9\\-3\\1
\end{tabular}\endgroup%
\kern3pt%
\begingroup \smaller\smaller\smaller\begin{tabular}{@{}c@{}}%
30\\-11\\1
\end{tabular}\endgroup%
{$\left.\llap{\phantom{%
\begingroup \smaller\smaller\smaller\begin{tabular}{@{}c@{}}%
0\\0\\0
\end{tabular}\endgroup%
}}\!\right]$}%
}%
\ifdim\wd\matricesbox>\halfwidth\myboxwidth=\hsize\else\myboxwidth=\halfwidth\fi
\vbox{%
\ifdim\myboxwidth=\hsize
\setbox\onelinebox=\hbox{%
\vbox{\hbox{%
$\Pi_{17,18}$ spans $L_{16.13}$%
}\hbox{%
$322232|23222322\slashthree22\rtimes D_{2}$%
}%
}%
\hfill\copy\matricesbox
}%
\ifdim\wd\onelinebox>\myboxwidth
\hbox to \myboxwidth{%
$\Pi_{17,18}$ spans $L_{16.13}$%
\hfil
$322232|23222322\slashthree22\rtimes D_{2}$%
}%
\box\matricesbox
\else
\hbox to \myboxwidth{%
\unhbox\onelinebox
}%
\fi
\else
\hbox to \myboxwidth{%
$\Pi_{17,18}$ spans $L_{16.13}$%
\hfil}%
\hbox to \myboxwidth{%
$322232|23222322\slashthree22\rtimes D_{2}$%
\hfil}%
\box\matricesbox
\fi
}%
\hfill\discretionary{}{}{}%
\setbox\matricesbox=\hbox{%
{$\left[\!\llap{\phantom{%
\begingroup \smaller\smaller\smaller\begin{tabular}{@{}c@{}}%
\phantom{0}\\\phantom{0}\\\phantom{0}
\end{tabular}\endgroup%
}}\right.$}%
\begingroup \smaller\smaller\smaller\begin{tabular}{@{}c@{}}%
-1\\\phantom{0}\\\phantom{0}
\end{tabular}\endgroup%
\kern3pt%
\begingroup \smaller\smaller\smaller\begin{tabular}{@{}c@{}}%
\phantom{0}\\45/2\\\phantom{0}
\end{tabular}\endgroup%
\kern3pt%
\begingroup \smaller\smaller\smaller\begin{tabular}{@{}c@{}}%
\phantom{0}\\\phantom{0}\\15/2
\end{tabular}\endgroup%
{$\left.\llap{\phantom{%
\begingroup \smaller\smaller\smaller\begin{tabular}{@{}c@{}}%
\phantom{0}\\\phantom{0}\\\phantom{0}
\end{tabular}\endgroup%
}}\!\right]$}%
{$\left[\!\llap{\phantom{%
\begingroup \smaller\smaller\smaller\begin{tabular}{@{}c@{}}%
0\\0\\0
\end{tabular}\endgroup%
}}\right.$}%
\begingroup \smaller\smaller\smaller\begin{tabular}{@{}c@{}}%
90\\19\\3
\end{tabular}\endgroup%
\kern3pt%
\begingroup \smaller\smaller\smaller\begin{tabular}{@{}c@{}}%
5\\1\\1
\end{tabular}\endgroup%
\kern3pt%
\begingroup \smaller\smaller\smaller\begin{tabular}{@{}c@{}}%
90\\11\\27
\end{tabular}\endgroup%
\kern3pt%
\begingroup \smaller\smaller\smaller\begin{tabular}{@{}c@{}}%
90\\8\\30
\end{tabular}\endgroup%
\kern3pt%
\begingroup \smaller\smaller\smaller\begin{tabular}{@{}c@{}}%
5\\0\\2
\end{tabular}\endgroup%
\kern3pt%
\begingroup \smaller\smaller\smaller\begin{tabular}{@{}c@{}}%
9\\-1\\3
\end{tabular}\endgroup%
\kern3pt%
\begingroup \smaller\smaller\smaller\begin{tabular}{@{}c@{}}%
30\\-5\\7
\end{tabular}\endgroup%
\kern3pt%
\begingroup \smaller\smaller\smaller\begin{tabular}{@{}c@{}}%
30\\-6\\4
\end{tabular}\endgroup%
\kern3pt%
\begingroup \smaller\smaller\smaller\begin{tabular}{@{}c@{}}%
9\\-2\\0
\end{tabular}\endgroup%
{$\left.\llap{\phantom{%
\begingroup \smaller\smaller\smaller\begin{tabular}{@{}c@{}}%
0\\0\\0
\end{tabular}\endgroup%
}}\!\right]$}%
}%
\ifdim\wd\matricesbox>\halfwidth\myboxwidth=\hsize\else\myboxwidth=\halfwidth\fi
\vbox{%
\ifdim\myboxwidth=\hsize
\setbox\onelinebox=\hbox{%
\vbox{\hbox{%
$\Pi_{17,19}$ spans $L_{16.13}$%
}\hbox{%
$3222322\slashthree22322232|2\rtimes D_{2}$%
}%
}%
\hfill\copy\matricesbox
}%
\ifdim\wd\onelinebox>\myboxwidth
\hbox to \myboxwidth{%
$\Pi_{17,19}$ spans $L_{16.13}$%
\hfil
$3222322\slashthree22322232|2\rtimes D_{2}$%
}%
\box\matricesbox
\else
\hbox to \myboxwidth{%
\unhbox\onelinebox
}%
\fi
\else
\hbox to \myboxwidth{%
$\Pi_{17,19}$ spans $L_{16.13}$%
\hfil}%
\hbox to \myboxwidth{%
$3222322\slashthree22322232|2\rtimes D_{2}$%
\hfil}%
\box\matricesbox
\fi
}%
\hfill\discretionary{}{}{}%
\setbox\matricesbox=\hbox{%
{$\left[\!\llap{\phantom{%
\begingroup \smaller\smaller\smaller\begin{tabular}{@{}c@{}}%
\phantom{0}\\\phantom{0}\\\phantom{0}
\end{tabular}\endgroup%
}}\right.$}%
\begingroup \smaller\smaller\smaller\begin{tabular}{@{}c@{}}%
-1\\\phantom{0}\\\phantom{0}
\end{tabular}\endgroup%
\kern3pt%
\begingroup \smaller\smaller\smaller\begin{tabular}{@{}c@{}}%
\phantom{0}\\15/2\\\phantom{0}
\end{tabular}\endgroup%
\kern3pt%
\begingroup \smaller\smaller\smaller\begin{tabular}{@{}c@{}}%
\phantom{0}\\\phantom{0}\\45/2
\end{tabular}\endgroup%
{$\left.\llap{\phantom{%
\begingroup \smaller\smaller\smaller\begin{tabular}{@{}c@{}}%
\phantom{0}\\\phantom{0}\\\phantom{0}
\end{tabular}\endgroup%
}}\!\right]$}%
{$\left[\!\llap{\phantom{%
\begingroup \smaller\smaller\smaller\begin{tabular}{@{}c@{}}%
0\\0\\0
\end{tabular}\endgroup%
}}\right.$}%
\begingroup \smaller\smaller\smaller\begin{tabular}{@{}c@{}}%
30\\11\\1
\end{tabular}\endgroup%
\kern3pt%
\begingroup \smaller\smaller\smaller\begin{tabular}{@{}c@{}}%
9\\3\\1
\end{tabular}\endgroup%
\kern3pt%
\begingroup \smaller\smaller\smaller\begin{tabular}{@{}c@{}}%
5\\1\\1
\end{tabular}\endgroup%
\kern3pt%
\begingroup \smaller\smaller\smaller\begin{tabular}{@{}c@{}}%
90\\3\\19
\end{tabular}\endgroup%
\kern3pt%
\begingroup \smaller\smaller\smaller\begin{tabular}{@{}c@{}}%
90\\-3\\19
\end{tabular}\endgroup%
\kern3pt%
\begingroup \smaller\smaller\smaller\begin{tabular}{@{}c@{}}%
5\\-1\\1
\end{tabular}\endgroup%
\kern3pt%
\begingroup \smaller\smaller\smaller\begin{tabular}{@{}c@{}}%
90\\-27\\11
\end{tabular}\endgroup%
\kern3pt%
\begingroup \smaller\smaller\smaller\begin{tabular}{@{}c@{}}%
90\\-30\\8
\end{tabular}\endgroup%
\kern3pt%
\begingroup \smaller\smaller\smaller\begin{tabular}{@{}c@{}}%
5\\-2\\0
\end{tabular}\endgroup%
{$\left.\llap{\phantom{%
\begingroup \smaller\smaller\smaller\begin{tabular}{@{}c@{}}%
0\\0\\0
\end{tabular}\endgroup%
}}\!\right]$}%
}%
\ifdim\wd\matricesbox>\halfwidth\myboxwidth=\hsize\else\myboxwidth=\halfwidth\fi
\vbox{%
\ifdim\myboxwidth=\hsize
\setbox\onelinebox=\hbox{%
\vbox{\hbox{%
$\Pi_{17,20}$ spans $L_{16.13}$%
}\hbox{%
$\slashthree22232232|23223222\rtimes D_{2}$%
}%
}%
\hfill\copy\matricesbox
}%
\ifdim\wd\onelinebox>\myboxwidth
\hbox to \myboxwidth{%
$\Pi_{17,20}$ spans $L_{16.13}$%
\hfil
$\slashthree22232232|23223222\rtimes D_{2}$%
}%
\box\matricesbox
\else
\hbox to \myboxwidth{%
\unhbox\onelinebox
}%
\fi
\else
\hbox to \myboxwidth{%
$\Pi_{17,20}$ spans $L_{16.13}$%
\hfil}%
\hbox to \myboxwidth{%
$\slashthree22232232|23223222\rtimes D_{2}$%
\hfil}%
\box\matricesbox
\fi
}%
\hfill\discretionary{}{}{}%
\setbox\matricesbox=\hbox{%
{$\left[\!\llap{\phantom{%
\begingroup \smaller\smaller\smaller\begin{tabular}{@{}c@{}}%
\phantom{0}\\\phantom{0}\\\phantom{0}
\end{tabular}\endgroup%
}}\right.$}%
\begingroup \smaller\smaller\smaller\begin{tabular}{@{}c@{}}%
-1\\\phantom{0}\\\phantom{0}
\end{tabular}\endgroup%
\kern3pt%
\begingroup \smaller\smaller\smaller\begin{tabular}{@{}c@{}}%
\phantom{0}\\21/2\\\phantom{0}
\end{tabular}\endgroup%
\kern3pt%
\begingroup \smaller\smaller\smaller\begin{tabular}{@{}c@{}}%
\phantom{0}\\\phantom{0}\\7/2
\end{tabular}\endgroup%
{$\left.\llap{\phantom{%
\begingroup \smaller\smaller\smaller\begin{tabular}{@{}c@{}}%
\phantom{0}\\\phantom{0}\\\phantom{0}
\end{tabular}\endgroup%
}}\!\right]$}%
{$\left[\!\llap{\phantom{%
\begingroup \smaller\smaller\smaller\begin{tabular}{@{}c@{}}%
0\\0\\0
\end{tabular}\endgroup%
}}\right.$}%
\begingroup \smaller\smaller\smaller\begin{tabular}{@{}c@{}}%
6\\-2\\0
\end{tabular}\endgroup%
\kern3pt%
\begingroup \smaller\smaller\smaller\begin{tabular}{@{}c@{}}%
7\\-2\\2
\end{tabular}\endgroup%
\kern3pt%
\begingroup \smaller\smaller\smaller\begin{tabular}{@{}c@{}}%
42\\-8\\18
\end{tabular}\endgroup%
\kern3pt%
\begingroup \smaller\smaller\smaller\begin{tabular}{@{}c@{}}%
42\\-5\\21
\end{tabular}\endgroup%
\kern3pt%
\begingroup \smaller\smaller\smaller\begin{tabular}{@{}c@{}}%
7\\0\\4
\end{tabular}\endgroup%
\kern3pt%
\begingroup \smaller\smaller\smaller\begin{tabular}{@{}c@{}}%
42\\5\\21
\end{tabular}\endgroup%
\kern3pt%
\begingroup \smaller\smaller\smaller\begin{tabular}{@{}c@{}}%
42\\8\\18
\end{tabular}\endgroup%
\kern3pt%
\begingroup \smaller\smaller\smaller\begin{tabular}{@{}c@{}}%
7\\2\\2
\end{tabular}\endgroup%
\kern3pt%
\begingroup \smaller\smaller\smaller\begin{tabular}{@{}c@{}}%
42\\13\\3
\end{tabular}\endgroup%
{$\left.\llap{\phantom{%
\begingroup \smaller\smaller\smaller\begin{tabular}{@{}c@{}}%
0\\0\\0
\end{tabular}\endgroup%
}}\!\right]$}%
}%
\ifdim\wd\matricesbox>\halfwidth\myboxwidth=\hsize\else\myboxwidth=\halfwidth\fi
\vbox{%
\ifdim\myboxwidth=\hsize
\setbox\onelinebox=\hbox{%
\vbox{\hbox{%
$\Pi_{17,21}$ spans $L_{22.4}$%
}\hbox{%
$322|22322322\slashthree22322\rtimes D_{2}$%
}%
}%
\hfill\copy\matricesbox
}%
\ifdim\wd\onelinebox>\myboxwidth
\hbox to \myboxwidth{%
$\Pi_{17,21}$ spans $L_{22.4}$%
\hfil
$322|22322322\slashthree22322\rtimes D_{2}$%
}%
\box\matricesbox
\else
\hbox to \myboxwidth{%
\unhbox\onelinebox
}%
\fi
\else
\hbox to \myboxwidth{%
$\Pi_{17,21}$ spans $L_{22.4}$%
\hfil}%
\hbox to \myboxwidth{%
$322|22322322\slashthree22322\rtimes D_{2}$%
\hfil}%
\box\matricesbox
\fi
}%
\hfill\discretionary{}{}{}%
\setbox\matricesbox=\hbox{%
{$\left[\!\llap{\phantom{%
\begingroup \smaller\smaller\smaller\begin{tabular}{@{}c@{}}%
\phantom{0}\\\phantom{0}\\\phantom{0}
\end{tabular}\endgroup%
}}\right.$}%
\begingroup \smaller\smaller\smaller\begin{tabular}{@{}c@{}}%
-1\\\phantom{0}\\\phantom{0}
\end{tabular}\endgroup%
\kern3pt%
\begingroup \smaller\smaller\smaller\begin{tabular}{@{}c@{}}%
\phantom{0}\\45/2\\\phantom{0}
\end{tabular}\endgroup%
\kern3pt%
\begingroup \smaller\smaller\smaller\begin{tabular}{@{}c@{}}%
\phantom{0}\\\phantom{0}\\15/2
\end{tabular}\endgroup%
{$\left.\llap{\phantom{%
\begingroup \smaller\smaller\smaller\begin{tabular}{@{}c@{}}%
\phantom{0}\\\phantom{0}\\\phantom{0}
\end{tabular}\endgroup%
}}\!\right]$}%
{$\left[\!\llap{\phantom{%
\begingroup \smaller\smaller\smaller\begin{tabular}{@{}c@{}}%
0\\0\\0
\end{tabular}\endgroup%
}}\right.$}%
\begingroup \smaller\smaller\smaller\begin{tabular}{@{}c@{}}%
9\\2\\0
\end{tabular}\endgroup%
\kern3pt%
\begingroup \smaller\smaller\smaller\begin{tabular}{@{}c@{}}%
5\\1\\-1
\end{tabular}\endgroup%
\kern3pt%
\begingroup \smaller\smaller\smaller\begin{tabular}{@{}c@{}}%
90\\11\\-27
\end{tabular}\endgroup%
\kern3pt%
\begingroup \smaller\smaller\smaller\begin{tabular}{@{}c@{}}%
90\\8\\-30
\end{tabular}\endgroup%
\kern3pt%
\begingroup \smaller\smaller\smaller\begin{tabular}{@{}c@{}}%
5\\0\\-2
\end{tabular}\endgroup%
\kern3pt%
\begingroup \smaller\smaller\smaller\begin{tabular}{@{}c@{}}%
90\\-8\\-30
\end{tabular}\endgroup%
\kern3pt%
\begingroup \smaller\smaller\smaller\begin{tabular}{@{}c@{}}%
90\\-11\\-27
\end{tabular}\endgroup%
\kern3pt%
\begingroup \smaller\smaller\smaller\begin{tabular}{@{}c@{}}%
5\\-1\\-1
\end{tabular}\endgroup%
\kern3pt%
\begingroup \smaller\smaller\smaller\begin{tabular}{@{}c@{}}%
90\\-19\\-3
\end{tabular}\endgroup%
{$\left.\llap{\phantom{%
\begingroup \smaller\smaller\smaller\begin{tabular}{@{}c@{}}%
0\\0\\0
\end{tabular}\endgroup%
}}\!\right]$}%
}%
\ifdim\wd\matricesbox>\halfwidth\myboxwidth=\hsize\else\myboxwidth=\halfwidth\fi
\vbox{%
\ifdim\myboxwidth=\hsize
\setbox\onelinebox=\hbox{%
\vbox{\hbox{%
$\Pi_{17,22}$ spans $L_{16.13}$%
}\hbox{%
$322|22322322\slashthree22322\rtimes D_{2}$%
}%
}%
\hfill\copy\matricesbox
}%
\ifdim\wd\onelinebox>\myboxwidth
\hbox to \myboxwidth{%
$\Pi_{17,22}$ spans $L_{16.13}$%
\hfil
$322|22322322\slashthree22322\rtimes D_{2}$%
}%
\box\matricesbox
\else
\hbox to \myboxwidth{%
\unhbox\onelinebox
}%
\fi
\else
\hbox to \myboxwidth{%
$\Pi_{17,22}$ spans $L_{16.13}$%
\hfil}%
\hbox to \myboxwidth{%
$322|22322322\slashthree22322\rtimes D_{2}$%
\hfil}%
\box\matricesbox
\fi
}%
\hfill\discretionary{}{}{}%
\setbox\matricesbox=\hbox{%
{$\left[\!\llap{\phantom{%
\begingroup \smaller\smaller\smaller
\endgroup%
}}\!\right]$}%
}%
\ifdim\wd\matricesbox>\halfwidth\myboxwidth=\hsize\else\myboxwidth=\halfwidth\fi
\vbox{%
\ifdim\myboxwidth=\hsize
\setbox\onelinebox=\hbox{%
\vbox{\hbox{%
$\Pi_{17,23}$ spans $L_{16.13}$%
}\hbox{%
$22222232236363632$%
}%
}%
\hfill\copy\matricesbox
}%
\ifdim\wd\onelinebox>\myboxwidth
\hbox to \myboxwidth{%
$\Pi_{17,23}$ spans $L_{16.13}$%
\hfil
$22222232236363632$%
}%
\box\matricesbox
\else
\hbox to \myboxwidth{%
\unhbox\onelinebox
}%
\fi
\else
\hbox to \myboxwidth{%
$\Pi_{17,23}$ spans $L_{16.13}$%
\hfil}%
\hbox to \myboxwidth{%
$22222232236363632$%
\hfil}%
\box\matricesbox
\fi
}%
\hfill\discretionary{}{}{}%
\setbox\matricesbox=\hbox{%
{$\left[\!\llap{\phantom{%
\begingroup \smaller\smaller\smaller
\endgroup%
}}\!\right]$}%
}%
\ifdim\wd\matricesbox>\halfwidth\myboxwidth=\hsize\else\myboxwidth=\halfwidth\fi
\vbox{%
\ifdim\myboxwidth=\hsize
\setbox\onelinebox=\hbox{%
\vbox{\hbox{%
$\Pi_{17,24}$ spans $L_{16.13}$%
}\hbox{%
$22222236322363632$%
}%
}%
\hfill\copy\matricesbox
}%
\ifdim\wd\onelinebox>\myboxwidth
\hbox to \myboxwidth{%
$\Pi_{17,24}$ spans $L_{16.13}$%
\hfil
$22222236322363632$%
}%
\box\matricesbox
\else
\hbox to \myboxwidth{%
\unhbox\onelinebox
}%
\fi
\else
\hbox to \myboxwidth{%
$\Pi_{17,24}$ spans $L_{16.13}$%
\hfil}%
\hbox to \myboxwidth{%
$22222236322363632$%
\hfil}%
\box\matricesbox
\fi
}%
\hfill\discretionary{}{}{}%
\setbox\matricesbox=\hbox{%
{$\left[\!\llap{\phantom{%
\begingroup \smaller\smaller\smaller
\endgroup%
}}\!\right]$}%
}%
\ifdim\wd\matricesbox>\halfwidth\myboxwidth=\hsize\else\myboxwidth=\halfwidth\fi
\vbox{%
\ifdim\myboxwidth=\hsize
\setbox\onelinebox=\hbox{%
\vbox{\hbox{%
$\Pi_{17,25}$ spans $L_{16.13}$%
}\hbox{%
$22222236363223632$%
}%
}%
\hfill\copy\matricesbox
}%
\ifdim\wd\onelinebox>\myboxwidth
\hbox to \myboxwidth{%
$\Pi_{17,25}$ spans $L_{16.13}$%
\hfil
$22222236363223632$%
}%
\box\matricesbox
\else
\hbox to \myboxwidth{%
\unhbox\onelinebox
}%
\fi
\else
\hbox to \myboxwidth{%
$\Pi_{17,25}$ spans $L_{16.13}$%
\hfil}%
\hbox to \myboxwidth{%
$22222236363223632$%
\hfil}%
\box\matricesbox
\fi
}%
\hfill\discretionary{}{}{}%
\setbox\matricesbox=\hbox{%
{$\left[\!\llap{\phantom{%
\begingroup \smaller\smaller\smaller
\endgroup%
}}\!\right]$}%
}%
\ifdim\wd\matricesbox>\halfwidth\myboxwidth=\hsize\else\myboxwidth=\halfwidth\fi
\vbox{%
\ifdim\myboxwidth=\hsize
\setbox\onelinebox=\hbox{%
\vbox{\hbox{%
$\Pi_{17,26}$ spans $L_{16.13}$%
}\hbox{%
$22222236363632232$%
}%
}%
\hfill\copy\matricesbox
}%
\ifdim\wd\onelinebox>\myboxwidth
\hbox to \myboxwidth{%
$\Pi_{17,26}$ spans $L_{16.13}$%
\hfil
$22222236363632232$%
}%
\box\matricesbox
\else
\hbox to \myboxwidth{%
\unhbox\onelinebox
}%
\fi
\else
\hbox to \myboxwidth{%
$\Pi_{17,26}$ spans $L_{16.13}$%
\hfil}%
\hbox to \myboxwidth{%
$22222236363632232$%
\hfil}%
\box\matricesbox
\fi
}%
\hfill\discretionary{}{}{}%
\setbox\matricesbox=\hbox{%
{$\left[\!\llap{\phantom{%
\begingroup \smaller\smaller\smaller
\endgroup%
}}\!\right]$}%
}%
\ifdim\wd\matricesbox>\halfwidth\myboxwidth=\hsize\else\myboxwidth=\halfwidth\fi
\vbox{%
\ifdim\myboxwidth=\hsize
\setbox\onelinebox=\hbox{%
\vbox{\hbox{%
$\Pi_{17,27}$ spans $L_{16.13}$%
}\hbox{%
$22222322236363632$%
}%
}%
\hfill\copy\matricesbox
}%
\ifdim\wd\onelinebox>\myboxwidth
\hbox to \myboxwidth{%
$\Pi_{17,27}$ spans $L_{16.13}$%
\hfil
$22222322236363632$%
}%
\box\matricesbox
\else
\hbox to \myboxwidth{%
\unhbox\onelinebox
}%
\fi
\else
\hbox to \myboxwidth{%
$\Pi_{17,27}$ spans $L_{16.13}$%
\hfil}%
\hbox to \myboxwidth{%
$22222322236363632$%
\hfil}%
\box\matricesbox
\fi
}%
\hfill\discretionary{}{}{}%
\setbox\matricesbox=\hbox{%
{$\left[\!\llap{\phantom{%
\begingroup \smaller\smaller\smaller
\endgroup%
}}\!\right]$}%
}%
\ifdim\wd\matricesbox>\halfwidth\myboxwidth=\hsize\else\myboxwidth=\halfwidth\fi
\vbox{%
\ifdim\myboxwidth=\hsize
\setbox\onelinebox=\hbox{%
\vbox{\hbox{%
$\Pi_{17,28}$ spans $L_{16.13}$%
}\hbox{%
$22222322322363632$%
}%
}%
\hfill\copy\matricesbox
}%
\ifdim\wd\onelinebox>\myboxwidth
\hbox to \myboxwidth{%
$\Pi_{17,28}$ spans $L_{16.13}$%
\hfil
$22222322322363632$%
}%
\box\matricesbox
\else
\hbox to \myboxwidth{%
\unhbox\onelinebox
}%
\fi
\else
\hbox to \myboxwidth{%
$\Pi_{17,28}$ spans $L_{16.13}$%
\hfil}%
\hbox to \myboxwidth{%
$22222322322363632$%
\hfil}%
\box\matricesbox
\fi
}%
\hfill\discretionary{}{}{}%
\setbox\matricesbox=\hbox{%
{$\left[\!\llap{\phantom{%
\begingroup \smaller\smaller\smaller
\endgroup%
}}\!\right]$}%
}%
\ifdim\wd\matricesbox>\halfwidth\myboxwidth=\hsize\else\myboxwidth=\halfwidth\fi
\vbox{%
\ifdim\myboxwidth=\hsize
\setbox\onelinebox=\hbox{%
\vbox{\hbox{%
$\Pi_{17,29}$ spans $L_{16.13}$%
}\hbox{%
$22222322363223632$%
}%
}%
\hfill\copy\matricesbox
}%
\ifdim\wd\onelinebox>\myboxwidth
\hbox to \myboxwidth{%
$\Pi_{17,29}$ spans $L_{16.13}$%
\hfil
$22222322363223632$%
}%
\box\matricesbox
\else
\hbox to \myboxwidth{%
\unhbox\onelinebox
}%
\fi
\else
\hbox to \myboxwidth{%
$\Pi_{17,29}$ spans $L_{16.13}$%
\hfil}%
\hbox to \myboxwidth{%
$22222322363223632$%
\hfil}%
\box\matricesbox
\fi
}%
\hfill\discretionary{}{}{}%
\setbox\matricesbox=\hbox{%
{$\left[\!\llap{\phantom{%
\begingroup \smaller\smaller\smaller
\endgroup%
}}\!\right]$}%
}%
\ifdim\wd\matricesbox>\halfwidth\myboxwidth=\hsize\else\myboxwidth=\halfwidth\fi
\vbox{%
\ifdim\myboxwidth=\hsize
\setbox\onelinebox=\hbox{%
\vbox{\hbox{%
$\Pi_{17,30}$ spans $L_{16.13}$%
}\hbox{%
$22222363222363632$%
}%
}%
\hfill\copy\matricesbox
}%
\ifdim\wd\onelinebox>\myboxwidth
\hbox to \myboxwidth{%
$\Pi_{17,30}$ spans $L_{16.13}$%
\hfil
$22222363222363632$%
}%
\box\matricesbox
\else
\hbox to \myboxwidth{%
\unhbox\onelinebox
}%
\fi
\else
\hbox to \myboxwidth{%
$\Pi_{17,30}$ spans $L_{16.13}$%
\hfil}%
\hbox to \myboxwidth{%
$22222363222363632$%
\hfil}%
\box\matricesbox
\fi
}%
\hfill\discretionary{}{}{}%
\setbox\matricesbox=\hbox{%
{$\left[\!\llap{\phantom{%
\begingroup \smaller\smaller\smaller
\endgroup%
}}\!\right]$}%
}%
\ifdim\wd\matricesbox>\halfwidth\myboxwidth=\hsize\else\myboxwidth=\halfwidth\fi
\vbox{%
\ifdim\myboxwidth=\hsize
\setbox\onelinebox=\hbox{%
\vbox{\hbox{%
$\Pi_{17,31}$ spans $L_{16.13}$%
}\hbox{%
$36322222232223636$%
}%
}%
\hfill\copy\matricesbox
}%
\ifdim\wd\onelinebox>\myboxwidth
\hbox to \myboxwidth{%
$\Pi_{17,31}$ spans $L_{16.13}$%
\hfil
$36322222232223636$%
}%
\box\matricesbox
\else
\hbox to \myboxwidth{%
\unhbox\onelinebox
}%
\fi
\else
\hbox to \myboxwidth{%
$\Pi_{17,31}$ spans $L_{16.13}$%
\hfil}%
\hbox to \myboxwidth{%
$36322222232223636$%
\hfil}%
\box\matricesbox
\fi
}%
\hfill\discretionary{}{}{}%
\setbox\matricesbox=\hbox{%
{$\left[\!\llap{\phantom{%
\begingroup \smaller\smaller\smaller
\endgroup%
}}\!\right]$}%
}%
\ifdim\wd\matricesbox>\halfwidth\myboxwidth=\hsize\else\myboxwidth=\halfwidth\fi
\vbox{%
\ifdim\myboxwidth=\hsize
\setbox\onelinebox=\hbox{%
\vbox{\hbox{%
$\Pi_{17,32}$ spans $L_{16.13}$%
}\hbox{%
$36322222232232236$%
}%
}%
\hfill\copy\matricesbox
}%
\ifdim\wd\onelinebox>\myboxwidth
\hbox to \myboxwidth{%
$\Pi_{17,32}$ spans $L_{16.13}$%
\hfil
$36322222232232236$%
}%
\box\matricesbox
\else
\hbox to \myboxwidth{%
\unhbox\onelinebox
}%
\fi
\else
\hbox to \myboxwidth{%
$\Pi_{17,32}$ spans $L_{16.13}$%
\hfil}%
\hbox to \myboxwidth{%
$36322222232232236$%
\hfil}%
\box\matricesbox
\fi
}%
\hfill\discretionary{}{}{}%
\setbox\matricesbox=\hbox{%
{$\left[\!\llap{\phantom{%
\begingroup \smaller\smaller\smaller
\endgroup%
}}\!\right]$}%
}%
\ifdim\wd\matricesbox>\halfwidth\myboxwidth=\hsize\else\myboxwidth=\halfwidth\fi
\vbox{%
\ifdim\myboxwidth=\hsize
\setbox\onelinebox=\hbox{%
\vbox{\hbox{%
$\Pi_{17,33}$ spans $L_{16.13}$%
}\hbox{%
$36322222232236322$%
}%
}%
\hfill\copy\matricesbox
}%
\ifdim\wd\onelinebox>\myboxwidth
\hbox to \myboxwidth{%
$\Pi_{17,33}$ spans $L_{16.13}$%
\hfil
$36322222232236322$%
}%
\box\matricesbox
\else
\hbox to \myboxwidth{%
\unhbox\onelinebox
}%
\fi
\else
\hbox to \myboxwidth{%
$\Pi_{17,33}$ spans $L_{16.13}$%
\hfil}%
\hbox to \myboxwidth{%
$36322222232236322$%
\hfil}%
\box\matricesbox
\fi
}%
\hfill\discretionary{}{}{}%
\setbox\matricesbox=\hbox{%
{$\left[\!\llap{\phantom{%
\begingroup \smaller\smaller\smaller
\endgroup%
}}\!\right]$}%
}%
\ifdim\wd\matricesbox>\halfwidth\myboxwidth=\hsize\else\myboxwidth=\halfwidth\fi
\vbox{%
\ifdim\myboxwidth=\hsize
\setbox\onelinebox=\hbox{%
\vbox{\hbox{%
$\Pi_{17,34}$ spans $L_{16.13}$%
}\hbox{%
$36322222236322236$%
}%
}%
\hfill\copy\matricesbox
}%
\ifdim\wd\onelinebox>\myboxwidth
\hbox to \myboxwidth{%
$\Pi_{17,34}$ spans $L_{16.13}$%
\hfil
$36322222236322236$%
}%
\box\matricesbox
\else
\hbox to \myboxwidth{%
\unhbox\onelinebox
}%
\fi
\else
\hbox to \myboxwidth{%
$\Pi_{17,34}$ spans $L_{16.13}$%
\hfil}%
\hbox to \myboxwidth{%
$36322222236322236$%
\hfil}%
\box\matricesbox
\fi
}%
\hfill\discretionary{}{}{}%
\setbox\matricesbox=\hbox{%
{$\left[\!\llap{\phantom{%
\begingroup \smaller\smaller\smaller
\endgroup%
}}\!\right]$}%
}%
\ifdim\wd\matricesbox>\halfwidth\myboxwidth=\hsize\else\myboxwidth=\halfwidth\fi
\vbox{%
\ifdim\myboxwidth=\hsize
\setbox\onelinebox=\hbox{%
\vbox{\hbox{%
$\Pi_{17,35}$ spans $L_{16.13}$%
}\hbox{%
$36322222322223636$%
}%
}%
\hfill\copy\matricesbox
}%
\ifdim\wd\onelinebox>\myboxwidth
\hbox to \myboxwidth{%
$\Pi_{17,35}$ spans $L_{16.13}$%
\hfil
$36322222322223636$%
}%
\box\matricesbox
\else
\hbox to \myboxwidth{%
\unhbox\onelinebox
}%
\fi
\else
\hbox to \myboxwidth{%
$\Pi_{17,35}$ spans $L_{16.13}$%
\hfil}%
\hbox to \myboxwidth{%
$36322222322223636$%
\hfil}%
\box\matricesbox
\fi
}%
\hfill\discretionary{}{}{}%
\setbox\matricesbox=\hbox{%
{$\left[\!\llap{\phantom{%
\begingroup \smaller\smaller\smaller
\endgroup%
}}\!\right]$}%
}%
\ifdim\wd\matricesbox>\halfwidth\myboxwidth=\hsize\else\myboxwidth=\halfwidth\fi
\vbox{%
\ifdim\myboxwidth=\hsize
\setbox\onelinebox=\hbox{%
\vbox{\hbox{%
$\Pi_{17,36}$ spans $L_{16.13}$%
}\hbox{%
$36322222322232236$%
}%
}%
\hfill\copy\matricesbox
}%
\ifdim\wd\onelinebox>\myboxwidth
\hbox to \myboxwidth{%
$\Pi_{17,36}$ spans $L_{16.13}$%
\hfil
$36322222322232236$%
}%
\box\matricesbox
\else
\hbox to \myboxwidth{%
\unhbox\onelinebox
}%
\fi
\else
\hbox to \myboxwidth{%
$\Pi_{17,36}$ spans $L_{16.13}$%
\hfil}%
\hbox to \myboxwidth{%
$36322222322232236$%
\hfil}%
\box\matricesbox
\fi
}%
\hfill\discretionary{}{}{}%
\setbox\matricesbox=\hbox{%
{$\left[\!\llap{\phantom{%
\begingroup \smaller\smaller\smaller
\endgroup%
}}\!\right]$}%
}%
\ifdim\wd\matricesbox>\halfwidth\myboxwidth=\hsize\else\myboxwidth=\halfwidth\fi
\vbox{%
\ifdim\myboxwidth=\hsize
\setbox\onelinebox=\hbox{%
\vbox{\hbox{%
$\Pi_{17,37}$ spans $L_{16.13}$%
}\hbox{%
$36322222322236322$%
}%
}%
\hfill\copy\matricesbox
}%
\ifdim\wd\onelinebox>\myboxwidth
\hbox to \myboxwidth{%
$\Pi_{17,37}$ spans $L_{16.13}$%
\hfil
$36322222322236322$%
}%
\box\matricesbox
\else
\hbox to \myboxwidth{%
\unhbox\onelinebox
}%
\fi
\else
\hbox to \myboxwidth{%
$\Pi_{17,37}$ spans $L_{16.13}$%
\hfil}%
\hbox to \myboxwidth{%
$36322222322236322$%
\hfil}%
\box\matricesbox
\fi
}%
\hfill\discretionary{}{}{}%
\setbox\matricesbox=\hbox{%
{$\left[\!\llap{\phantom{%
\begingroup \smaller\smaller\smaller
\endgroup%
}}\!\right]$}%
}%
\ifdim\wd\matricesbox>\halfwidth\myboxwidth=\hsize\else\myboxwidth=\halfwidth\fi
\vbox{%
\ifdim\myboxwidth=\hsize
\setbox\onelinebox=\hbox{%
\vbox{\hbox{%
$\Pi_{17,38}$ spans $L_{16.13}$%
}\hbox{%
$36322222322322236$%
}%
}%
\hfill\copy\matricesbox
}%
\ifdim\wd\onelinebox>\myboxwidth
\hbox to \myboxwidth{%
$\Pi_{17,38}$ spans $L_{16.13}$%
\hfil
$36322222322322236$%
}%
\box\matricesbox
\else
\hbox to \myboxwidth{%
\unhbox\onelinebox
}%
\fi
\else
\hbox to \myboxwidth{%
$\Pi_{17,38}$ spans $L_{16.13}$%
\hfil}%
\hbox to \myboxwidth{%
$36322222322322236$%
\hfil}%
\box\matricesbox
\fi
}%
\hfill\discretionary{}{}{}%
\setbox\matricesbox=\hbox{%
{$\left[\!\llap{\phantom{%
\begingroup \smaller\smaller\smaller
\endgroup%
}}\!\right]$}%
}%
\ifdim\wd\matricesbox>\halfwidth\myboxwidth=\hsize\else\myboxwidth=\halfwidth\fi
\vbox{%
\ifdim\myboxwidth=\hsize
\setbox\onelinebox=\hbox{%
\vbox{\hbox{%
$\Pi_{17,39}$ spans $L_{16.13}$%
}\hbox{%
$36322222322322322$%
}%
}%
\hfill\copy\matricesbox
}%
\ifdim\wd\onelinebox>\myboxwidth
\hbox to \myboxwidth{%
$\Pi_{17,39}$ spans $L_{16.13}$%
\hfil
$36322222322322322$%
}%
\box\matricesbox
\else
\hbox to \myboxwidth{%
\unhbox\onelinebox
}%
\fi
\else
\hbox to \myboxwidth{%
$\Pi_{17,39}$ spans $L_{16.13}$%
\hfil}%
\hbox to \myboxwidth{%
$36322222322322322$%
\hfil}%
\box\matricesbox
\fi
}%
\hfill\discretionary{}{}{}%
\setbox\matricesbox=\hbox{%
{$\left[\!\llap{\phantom{%
\begingroup \smaller\smaller\smaller
\endgroup%
}}\!\right]$}%
}%
\ifdim\wd\matricesbox>\halfwidth\myboxwidth=\hsize\else\myboxwidth=\halfwidth\fi
\vbox{%
\ifdim\myboxwidth=\hsize
\setbox\onelinebox=\hbox{%
\vbox{\hbox{%
$\Pi_{17,40}$ spans $L_{16.13}$%
}\hbox{%
$36322222322363222$%
}%
}%
\hfill\copy\matricesbox
}%
\ifdim\wd\onelinebox>\myboxwidth
\hbox to \myboxwidth{%
$\Pi_{17,40}$ spans $L_{16.13}$%
\hfil
$36322222322363222$%
}%
\box\matricesbox
\else
\hbox to \myboxwidth{%
\unhbox\onelinebox
}%
\fi
\else
\hbox to \myboxwidth{%
$\Pi_{17,40}$ spans $L_{16.13}$%
\hfil}%
\hbox to \myboxwidth{%
$36322222322363222$%
\hfil}%
\box\matricesbox
\fi
}%
\hfill\discretionary{}{}{}%
\setbox\matricesbox=\hbox{%
{$\left[\!\llap{\phantom{%
\begingroup \smaller\smaller\smaller
\endgroup%
}}\!\right]$}%
}%
\ifdim\wd\matricesbox>\halfwidth\myboxwidth=\hsize\else\myboxwidth=\halfwidth\fi
\vbox{%
\ifdim\myboxwidth=\hsize
\setbox\onelinebox=\hbox{%
\vbox{\hbox{%
$\Pi_{17,41}$ spans $L_{16.13}$%
}\hbox{%
$36322222363222236$%
}%
}%
\hfill\copy\matricesbox
}%
\ifdim\wd\onelinebox>\myboxwidth
\hbox to \myboxwidth{%
$\Pi_{17,41}$ spans $L_{16.13}$%
\hfil
$36322222363222236$%
}%
\box\matricesbox
\else
\hbox to \myboxwidth{%
\unhbox\onelinebox
}%
\fi
\else
\hbox to \myboxwidth{%
$\Pi_{17,41}$ spans $L_{16.13}$%
\hfil}%
\hbox to \myboxwidth{%
$36322222363222236$%
\hfil}%
\box\matricesbox
\fi
}%
\hfill\discretionary{}{}{}%
\setbox\matricesbox=\hbox{%
{$\left[\!\llap{\phantom{%
\begingroup \smaller\smaller\smaller
\endgroup%
}}\!\right]$}%
}%
\ifdim\wd\matricesbox>\halfwidth\myboxwidth=\hsize\else\myboxwidth=\halfwidth\fi
\vbox{%
\ifdim\myboxwidth=\hsize
\setbox\onelinebox=\hbox{%
\vbox{\hbox{%
$\Pi_{17,42}$ spans $L_{16.13}$%
}\hbox{%
$36322222363222322$%
}%
}%
\hfill\copy\matricesbox
}%
\ifdim\wd\onelinebox>\myboxwidth
\hbox to \myboxwidth{%
$\Pi_{17,42}$ spans $L_{16.13}$%
\hfil
$36322222363222322$%
}%
\box\matricesbox
\else
\hbox to \myboxwidth{%
\unhbox\onelinebox
}%
\fi
\else
\hbox to \myboxwidth{%
$\Pi_{17,42}$ spans $L_{16.13}$%
\hfil}%
\hbox to \myboxwidth{%
$36322222363222322$%
\hfil}%
\box\matricesbox
\fi
}%
\hfill\discretionary{}{}{}%
\setbox\matricesbox=\hbox{%
{$\left[\!\llap{\phantom{%
\begingroup \smaller\smaller\smaller
\endgroup%
}}\!\right]$}%
}%
\ifdim\wd\matricesbox>\halfwidth\myboxwidth=\hsize\else\myboxwidth=\halfwidth\fi
\vbox{%
\ifdim\myboxwidth=\hsize
\setbox\onelinebox=\hbox{%
\vbox{\hbox{%
$\Pi_{17,43}$ spans $L_{16.13}$%
}\hbox{%
$36322222363223222$%
}%
}%
\hfill\copy\matricesbox
}%
\ifdim\wd\onelinebox>\myboxwidth
\hbox to \myboxwidth{%
$\Pi_{17,43}$ spans $L_{16.13}$%
\hfil
$36322222363223222$%
}%
\box\matricesbox
\else
\hbox to \myboxwidth{%
\unhbox\onelinebox
}%
\fi
\else
\hbox to \myboxwidth{%
$\Pi_{17,43}$ spans $L_{16.13}$%
\hfil}%
\hbox to \myboxwidth{%
$36322222363223222$%
\hfil}%
\box\matricesbox
\fi
}%
\hfill\discretionary{}{}{}%
\setbox\matricesbox=\hbox{%
{$\left[\!\llap{\phantom{%
\begingroup \smaller\smaller\smaller
\endgroup%
}}\!\right]$}%
}%
\ifdim\wd\matricesbox>\halfwidth\myboxwidth=\hsize\else\myboxwidth=\halfwidth\fi
\vbox{%
\ifdim\myboxwidth=\hsize
\setbox\onelinebox=\hbox{%
\vbox{\hbox{%
$\Pi_{17,44}$ spans $L_{16.13}$%
}\hbox{%
$36322222363632222$%
}%
}%
\hfill\copy\matricesbox
}%
\ifdim\wd\onelinebox>\myboxwidth
\hbox to \myboxwidth{%
$\Pi_{17,44}$ spans $L_{16.13}$%
\hfil
$36322222363632222$%
}%
\box\matricesbox
\else
\hbox to \myboxwidth{%
\unhbox\onelinebox
}%
\fi
\else
\hbox to \myboxwidth{%
$\Pi_{17,44}$ spans $L_{16.13}$%
\hfil}%
\hbox to \myboxwidth{%
$36322222363632222$%
\hfil}%
\box\matricesbox
\fi
}%
\hfill\discretionary{}{}{}%
\setbox\matricesbox=\hbox{%
{$\left[\!\llap{\phantom{%
\begingroup \smaller\smaller\smaller
\endgroup%
}}\!\right]$}%
}%
\ifdim\wd\matricesbox>\halfwidth\myboxwidth=\hsize\else\myboxwidth=\halfwidth\fi
\vbox{%
\ifdim\myboxwidth=\hsize
\setbox\onelinebox=\hbox{%
\vbox{\hbox{%
$\Pi_{17,45}$ spans $L_{16.13}$%
}\hbox{%
$36322223222223636$%
}%
}%
\hfill\copy\matricesbox
}%
\ifdim\wd\onelinebox>\myboxwidth
\hbox to \myboxwidth{%
$\Pi_{17,45}$ spans $L_{16.13}$%
\hfil
$36322223222223636$%
}%
\box\matricesbox
\else
\hbox to \myboxwidth{%
\unhbox\onelinebox
}%
\fi
\else
\hbox to \myboxwidth{%
$\Pi_{17,45}$ spans $L_{16.13}$%
\hfil}%
\hbox to \myboxwidth{%
$36322223222223636$%
\hfil}%
\box\matricesbox
\fi
}%
\hfill\discretionary{}{}{}%
\setbox\matricesbox=\hbox{%
{$\left[\!\llap{\phantom{%
\begingroup \smaller\smaller\smaller
\endgroup%
}}\!\right]$}%
}%
\ifdim\wd\matricesbox>\halfwidth\myboxwidth=\hsize\else\myboxwidth=\halfwidth\fi
\vbox{%
\ifdim\myboxwidth=\hsize
\setbox\onelinebox=\hbox{%
\vbox{\hbox{%
$\Pi_{17,46}$ spans $L_{16.13}$%
}\hbox{%
$36322223222232236$%
}%
}%
\hfill\copy\matricesbox
}%
\ifdim\wd\onelinebox>\myboxwidth
\hbox to \myboxwidth{%
$\Pi_{17,46}$ spans $L_{16.13}$%
\hfil
$36322223222232236$%
}%
\box\matricesbox
\else
\hbox to \myboxwidth{%
\unhbox\onelinebox
}%
\fi
\else
\hbox to \myboxwidth{%
$\Pi_{17,46}$ spans $L_{16.13}$%
\hfil}%
\hbox to \myboxwidth{%
$36322223222232236$%
\hfil}%
\box\matricesbox
\fi
}%
\hfill\discretionary{}{}{}%
\setbox\matricesbox=\hbox{%
{$\left[\!\llap{\phantom{%
\begingroup \smaller\smaller\smaller
\endgroup%
}}\!\right]$}%
}%
\ifdim\wd\matricesbox>\halfwidth\myboxwidth=\hsize\else\myboxwidth=\halfwidth\fi
\vbox{%
\ifdim\myboxwidth=\hsize
\setbox\onelinebox=\hbox{%
\vbox{\hbox{%
$\Pi_{17,47}$ spans $L_{16.13}$%
}\hbox{%
$36322223222322236$%
}%
}%
\hfill\copy\matricesbox
}%
\ifdim\wd\onelinebox>\myboxwidth
\hbox to \myboxwidth{%
$\Pi_{17,47}$ spans $L_{16.13}$%
\hfil
$36322223222322236$%
}%
\box\matricesbox
\else
\hbox to \myboxwidth{%
\unhbox\onelinebox
}%
\fi
\else
\hbox to \myboxwidth{%
$\Pi_{17,47}$ spans $L_{16.13}$%
\hfil}%
\hbox to \myboxwidth{%
$36322223222322236$%
\hfil}%
\box\matricesbox
\fi
}%
\hfill\discretionary{}{}{}%
\setbox\matricesbox=\hbox{%
{$\left[\!\llap{\phantom{%
\begingroup \smaller\smaller\smaller
\endgroup%
}}\!\right]$}%
}%
\ifdim\wd\matricesbox>\halfwidth\myboxwidth=\hsize\else\myboxwidth=\halfwidth\fi
\vbox{%
\ifdim\myboxwidth=\hsize
\setbox\onelinebox=\hbox{%
\vbox{\hbox{%
$\Pi_{17,48}$ spans $L_{16.13}$%
}\hbox{%
$36322223222322322$%
}%
}%
\hfill\copy\matricesbox
}%
\ifdim\wd\onelinebox>\myboxwidth
\hbox to \myboxwidth{%
$\Pi_{17,48}$ spans $L_{16.13}$%
\hfil
$36322223222322322$%
}%
\box\matricesbox
\else
\hbox to \myboxwidth{%
\unhbox\onelinebox
}%
\fi
\else
\hbox to \myboxwidth{%
$\Pi_{17,48}$ spans $L_{16.13}$%
\hfil}%
\hbox to \myboxwidth{%
$36322223222322322$%
\hfil}%
\box\matricesbox
\fi
}%
\hfill\discretionary{}{}{}%
\setbox\matricesbox=\hbox{%
{$\left[\!\llap{\phantom{%
\begingroup \smaller\smaller\smaller
\endgroup%
}}\!\right]$}%
}%
\ifdim\wd\matricesbox>\halfwidth\myboxwidth=\hsize\else\myboxwidth=\halfwidth\fi
\vbox{%
\ifdim\myboxwidth=\hsize
\setbox\onelinebox=\hbox{%
\vbox{\hbox{%
$\Pi_{17,49}$ spans $L_{16.13}$%
}\hbox{%
$36322223222363222$%
}%
}%
\hfill\copy\matricesbox
}%
\ifdim\wd\onelinebox>\myboxwidth
\hbox to \myboxwidth{%
$\Pi_{17,49}$ spans $L_{16.13}$%
\hfil
$36322223222363222$%
}%
\box\matricesbox
\else
\hbox to \myboxwidth{%
\unhbox\onelinebox
}%
\fi
\else
\hbox to \myboxwidth{%
$\Pi_{17,49}$ spans $L_{16.13}$%
\hfil}%
\hbox to \myboxwidth{%
$36322223222363222$%
\hfil}%
\box\matricesbox
\fi
}%
\hfill\discretionary{}{}{}%
\setbox\matricesbox=\hbox{%
{$\left[\!\llap{\phantom{%
\begingroup \smaller\smaller\smaller
\endgroup%
}}\!\right]$}%
}%
\ifdim\wd\matricesbox>\halfwidth\myboxwidth=\hsize\else\myboxwidth=\halfwidth\fi
\vbox{%
\ifdim\myboxwidth=\hsize
\setbox\onelinebox=\hbox{%
\vbox{\hbox{%
$\Pi_{17,50}$ spans $L_{16.13}$%
}\hbox{%
$36322223223222322$%
}%
}%
\hfill\copy\matricesbox
}%
\ifdim\wd\onelinebox>\myboxwidth
\hbox to \myboxwidth{%
$\Pi_{17,50}$ spans $L_{16.13}$%
\hfil
$36322223223222322$%
}%
\box\matricesbox
\else
\hbox to \myboxwidth{%
\unhbox\onelinebox
}%
\fi
\else
\hbox to \myboxwidth{%
$\Pi_{17,50}$ spans $L_{16.13}$%
\hfil}%
\hbox to \myboxwidth{%
$36322223223222322$%
\hfil}%
\box\matricesbox
\fi
}%
\hfill\discretionary{}{}{}%
\setbox\matricesbox=\hbox{%
{$\left[\!\llap{\phantom{%
\begingroup \smaller\smaller\smaller
\endgroup%
}}\!\right]$}%
}%
\ifdim\wd\matricesbox>\halfwidth\myboxwidth=\hsize\else\myboxwidth=\halfwidth\fi
\vbox{%
\ifdim\myboxwidth=\hsize
\setbox\onelinebox=\hbox{%
\vbox{\hbox{%
$\Pi_{17,51}$ spans $L_{16.13}$%
}\hbox{%
$36322223223223222$%
}%
}%
\hfill\copy\matricesbox
}%
\ifdim\wd\onelinebox>\myboxwidth
\hbox to \myboxwidth{%
$\Pi_{17,51}$ spans $L_{16.13}$%
\hfil
$36322223223223222$%
}%
\box\matricesbox
\else
\hbox to \myboxwidth{%
\unhbox\onelinebox
}%
\fi
\else
\hbox to \myboxwidth{%
$\Pi_{17,51}$ spans $L_{16.13}$%
\hfil}%
\hbox to \myboxwidth{%
$36322223223223222$%
\hfil}%
\box\matricesbox
\fi
}%
\hfill\discretionary{}{}{}%
\setbox\matricesbox=\hbox{%
{$\left[\!\llap{\phantom{%
\begingroup \smaller\smaller\smaller
\endgroup%
}}\!\right]$}%
}%
\ifdim\wd\matricesbox>\halfwidth\myboxwidth=\hsize\else\myboxwidth=\halfwidth\fi
\vbox{%
\ifdim\myboxwidth=\hsize
\setbox\onelinebox=\hbox{%
\vbox{\hbox{%
$\Pi_{17,52}$ spans $L_{16.13}$%
}\hbox{%
$36322223223632222$%
}%
}%
\hfill\copy\matricesbox
}%
\ifdim\wd\onelinebox>\myboxwidth
\hbox to \myboxwidth{%
$\Pi_{17,52}$ spans $L_{16.13}$%
\hfil
$36322223223632222$%
}%
\box\matricesbox
\else
\hbox to \myboxwidth{%
\unhbox\onelinebox
}%
\fi
\else
\hbox to \myboxwidth{%
$\Pi_{17,52}$ spans $L_{16.13}$%
\hfil}%
\hbox to \myboxwidth{%
$36322223223632222$%
\hfil}%
\box\matricesbox
\fi
}%
\hfill\discretionary{}{}{}%
\setbox\matricesbox=\hbox{%
{$\left[\!\llap{\phantom{%
\begingroup \smaller\smaller\smaller
\endgroup%
}}\!\right]$}%
}%
\ifdim\wd\matricesbox>\halfwidth\myboxwidth=\hsize\else\myboxwidth=\halfwidth\fi
\vbox{%
\ifdim\myboxwidth=\hsize
\setbox\onelinebox=\hbox{%
\vbox{\hbox{%
$\Pi_{17,53}$ spans $L_{16.13}$%
}\hbox{%
$36322223632222322$%
}%
}%
\hfill\copy\matricesbox
}%
\ifdim\wd\onelinebox>\myboxwidth
\hbox to \myboxwidth{%
$\Pi_{17,53}$ spans $L_{16.13}$%
\hfil
$36322223632222322$%
}%
\box\matricesbox
\else
\hbox to \myboxwidth{%
\unhbox\onelinebox
}%
\fi
\else
\hbox to \myboxwidth{%
$\Pi_{17,53}$ spans $L_{16.13}$%
\hfil}%
\hbox to \myboxwidth{%
$36322223632222322$%
\hfil}%
\box\matricesbox
\fi
}%
\hfill\discretionary{}{}{}%
\setbox\matricesbox=\hbox{%
{$\left[\!\llap{\phantom{%
\begingroup \smaller\smaller\smaller
\endgroup%
}}\!\right]$}%
}%
\ifdim\wd\matricesbox>\halfwidth\myboxwidth=\hsize\else\myboxwidth=\halfwidth\fi
\vbox{%
\ifdim\myboxwidth=\hsize
\setbox\onelinebox=\hbox{%
\vbox{\hbox{%
$\Pi_{17,54}$ spans $L_{16.13}$%
}\hbox{%
$36322232222232236$%
}%
}%
\hfill\copy\matricesbox
}%
\ifdim\wd\onelinebox>\myboxwidth
\hbox to \myboxwidth{%
$\Pi_{17,54}$ spans $L_{16.13}$%
\hfil
$36322232222232236$%
}%
\box\matricesbox
\else
\hbox to \myboxwidth{%
\unhbox\onelinebox
}%
\fi
\else
\hbox to \myboxwidth{%
$\Pi_{17,54}$ spans $L_{16.13}$%
\hfil}%
\hbox to \myboxwidth{%
$36322232222232236$%
\hfil}%
\box\matricesbox
\fi
}%
\hfill\discretionary{}{}{}%
\setbox\matricesbox=\hbox{%
{$\left[\!\llap{\phantom{%
\begingroup \smaller\smaller\smaller
\endgroup%
}}\!\right]$}%
}%
\ifdim\wd\matricesbox>\halfwidth\myboxwidth=\hsize\else\myboxwidth=\halfwidth\fi
\vbox{%
\ifdim\myboxwidth=\hsize
\setbox\onelinebox=\hbox{%
\vbox{\hbox{%
$\Pi_{17,55}$ spans $L_{16.13}$%
}\hbox{%
$36322232222322322$%
}%
}%
\hfill\copy\matricesbox
}%
\ifdim\wd\onelinebox>\myboxwidth
\hbox to \myboxwidth{%
$\Pi_{17,55}$ spans $L_{16.13}$%
\hfil
$36322232222322322$%
}%
\box\matricesbox
\else
\hbox to \myboxwidth{%
\unhbox\onelinebox
}%
\fi
\else
\hbox to \myboxwidth{%
$\Pi_{17,55}$ spans $L_{16.13}$%
\hfil}%
\hbox to \myboxwidth{%
$36322232222322322$%
\hfil}%
\box\matricesbox
\fi
}%
\hfill\discretionary{}{}{}%
\setbox\matricesbox=\hbox{%
{$\left[\!\llap{\phantom{%
\begingroup \smaller\smaller\smaller
\endgroup%
}}\!\right]$}%
}%
\ifdim\wd\matricesbox>\halfwidth\myboxwidth=\hsize\else\myboxwidth=\halfwidth\fi
\vbox{%
\ifdim\myboxwidth=\hsize
\setbox\onelinebox=\hbox{%
\vbox{\hbox{%
$\Pi_{17,56}$ spans $L_{16.13}$%
}\hbox{%
$36322232222363222$%
}%
}%
\hfill\copy\matricesbox
}%
\ifdim\wd\onelinebox>\myboxwidth
\hbox to \myboxwidth{%
$\Pi_{17,56}$ spans $L_{16.13}$%
\hfil
$36322232222363222$%
}%
\box\matricesbox
\else
\hbox to \myboxwidth{%
\unhbox\onelinebox
}%
\fi
\else
\hbox to \myboxwidth{%
$\Pi_{17,56}$ spans $L_{16.13}$%
\hfil}%
\hbox to \myboxwidth{%
$36322232222363222$%
\hfil}%
\box\matricesbox
\fi
}%
\hfill\discretionary{}{}{}%
\setbox\matricesbox=\hbox{%
{$\left[\!\llap{\phantom{%
\begingroup \smaller\smaller\smaller
\endgroup%
}}\!\right]$}%
}%
\ifdim\wd\matricesbox>\halfwidth\myboxwidth=\hsize\else\myboxwidth=\halfwidth\fi
\vbox{%
\ifdim\myboxwidth=\hsize
\setbox\onelinebox=\hbox{%
\vbox{\hbox{%
$\Pi_{17,57}$ spans $L_{16.13}$%
}\hbox{%
$36322232223222322$%
}%
}%
\hfill\copy\matricesbox
}%
\ifdim\wd\onelinebox>\myboxwidth
\hbox to \myboxwidth{%
$\Pi_{17,57}$ spans $L_{16.13}$%
\hfil
$36322232223222322$%
}%
\box\matricesbox
\else
\hbox to \myboxwidth{%
\unhbox\onelinebox
}%
\fi
\else
\hbox to \myboxwidth{%
$\Pi_{17,57}$ spans $L_{16.13}$%
\hfil}%
\hbox to \myboxwidth{%
$36322232223222322$%
\hfil}%
\box\matricesbox
\fi
}%
\hfill\discretionary{}{}{}%
\setbox\matricesbox=\hbox{%
{$\left[\!\llap{\phantom{%
\begingroup \smaller\smaller\smaller
\endgroup%
}}\!\right]$}%
}%
\ifdim\wd\matricesbox>\halfwidth\myboxwidth=\hsize\else\myboxwidth=\halfwidth\fi
\vbox{%
\ifdim\myboxwidth=\hsize
\setbox\onelinebox=\hbox{%
\vbox{\hbox{%
$\Pi_{17,58}$ spans $L_{16.13}$%
}\hbox{%
$36322232223223222$%
}%
}%
\hfill\copy\matricesbox
}%
\ifdim\wd\onelinebox>\myboxwidth
\hbox to \myboxwidth{%
$\Pi_{17,58}$ spans $L_{16.13}$%
\hfil
$36322232223223222$%
}%
\box\matricesbox
\else
\hbox to \myboxwidth{%
\unhbox\onelinebox
}%
\fi
\else
\hbox to \myboxwidth{%
$\Pi_{17,58}$ spans $L_{16.13}$%
\hfil}%
\hbox to \myboxwidth{%
$36322232223223222$%
\hfil}%
\box\matricesbox
\fi
}%
\hfill\discretionary{}{}{}%
\setbox\matricesbox=\hbox{%
{$\left[\!\llap{\phantom{%
\begingroup \smaller\smaller\smaller
\endgroup%
}}\!\right]$}%
}%
\ifdim\wd\matricesbox>\halfwidth\myboxwidth=\hsize\else\myboxwidth=\halfwidth\fi
\vbox{%
\ifdim\myboxwidth=\hsize
\setbox\onelinebox=\hbox{%
\vbox{\hbox{%
$\Pi_{17,59}$ spans $L_{16.13}$%
}\hbox{%
$36322232232222322$%
}%
}%
\hfill\copy\matricesbox
}%
\ifdim\wd\onelinebox>\myboxwidth
\hbox to \myboxwidth{%
$\Pi_{17,59}$ spans $L_{16.13}$%
\hfil
$36322232232222322$%
}%
\box\matricesbox
\else
\hbox to \myboxwidth{%
\unhbox\onelinebox
}%
\fi
\else
\hbox to \myboxwidth{%
$\Pi_{17,59}$ spans $L_{16.13}$%
\hfil}%
\hbox to \myboxwidth{%
$36322232232222322$%
\hfil}%
\box\matricesbox
\fi
}%
\hfill\discretionary{}{}{}%
\setbox\matricesbox=\hbox{%
{$\left[\!\llap{\phantom{%
\begingroup \smaller\smaller\smaller
\endgroup%
}}\!\right]$}%
}%
\ifdim\wd\matricesbox>\halfwidth\myboxwidth=\hsize\else\myboxwidth=\halfwidth\fi
\vbox{%
\ifdim\myboxwidth=\hsize
\setbox\onelinebox=\hbox{%
\vbox{\hbox{%
$\Pi_{17,60}$ spans $L_{16.13}$%
}\hbox{%
$36322232232223222$%
}%
}%
\hfill\copy\matricesbox
}%
\ifdim\wd\onelinebox>\myboxwidth
\hbox to \myboxwidth{%
$\Pi_{17,60}$ spans $L_{16.13}$%
\hfil
$36322232232223222$%
}%
\box\matricesbox
\else
\hbox to \myboxwidth{%
\unhbox\onelinebox
}%
\fi
\else
\hbox to \myboxwidth{%
$\Pi_{17,60}$ spans $L_{16.13}$%
\hfil}%
\hbox to \myboxwidth{%
$36322232232223222$%
\hfil}%
\box\matricesbox
\fi
}%
\hfill\discretionary{}{}{}%
\setbox\matricesbox=\hbox{%
{$\left[\!\llap{\phantom{%
\begingroup \smaller\smaller\smaller
\endgroup%
}}\!\right]$}%
}%
\ifdim\wd\matricesbox>\halfwidth\myboxwidth=\hsize\else\myboxwidth=\halfwidth\fi
\vbox{%
\ifdim\myboxwidth=\hsize
\setbox\onelinebox=\hbox{%
\vbox{\hbox{%
$\Pi_{17,61}$ spans $L_{16.13}$%
}\hbox{%
$36322232232232222$%
}%
}%
\hfill\copy\matricesbox
}%
\ifdim\wd\onelinebox>\myboxwidth
\hbox to \myboxwidth{%
$\Pi_{17,61}$ spans $L_{16.13}$%
\hfil
$36322232232232222$%
}%
\box\matricesbox
\else
\hbox to \myboxwidth{%
\unhbox\onelinebox
}%
\fi
\else
\hbox to \myboxwidth{%
$\Pi_{17,61}$ spans $L_{16.13}$%
\hfil}%
\hbox to \myboxwidth{%
$36322232232232222$%
\hfil}%
\box\matricesbox
\fi
}%
\hfill\discretionary{}{}{}%
\setbox\matricesbox=\hbox{%
{$\left[\!\llap{\phantom{%
\begingroup \smaller\smaller\smaller
\endgroup%
}}\!\right]$}%
}%
\ifdim\wd\matricesbox>\halfwidth\myboxwidth=\hsize\else\myboxwidth=\halfwidth\fi
\vbox{%
\ifdim\myboxwidth=\hsize
\setbox\onelinebox=\hbox{%
\vbox{\hbox{%
$\Pi_{17,62}$ spans $L_{16.13}$%
}\hbox{%
$36322236322322222$%
}%
}%
\hfill\copy\matricesbox
}%
\ifdim\wd\onelinebox>\myboxwidth
\hbox to \myboxwidth{%
$\Pi_{17,62}$ spans $L_{16.13}$%
\hfil
$36322236322322222$%
}%
\box\matricesbox
\else
\hbox to \myboxwidth{%
\unhbox\onelinebox
}%
\fi
\else
\hbox to \myboxwidth{%
$\Pi_{17,62}$ spans $L_{16.13}$%
\hfil}%
\hbox to \myboxwidth{%
$36322236322322222$%
\hfil}%
\box\matricesbox
\fi
}%
\hfill\discretionary{}{}{}%
\setbox\matricesbox=\hbox{%
{$\left[\!\llap{\phantom{%
\begingroup \smaller\smaller\smaller
\endgroup%
}}\!\right]$}%
}%
\ifdim\wd\matricesbox>\halfwidth\myboxwidth=\hsize\else\myboxwidth=\halfwidth\fi
\vbox{%
\ifdim\myboxwidth=\hsize
\setbox\onelinebox=\hbox{%
\vbox{\hbox{%
$\Pi_{17,63}$ spans $L_{16.13}$%
}\hbox{%
$36322322222322322$%
}%
}%
\hfill\copy\matricesbox
}%
\ifdim\wd\onelinebox>\myboxwidth
\hbox to \myboxwidth{%
$\Pi_{17,63}$ spans $L_{16.13}$%
\hfil
$36322322222322322$%
}%
\box\matricesbox
\else
\hbox to \myboxwidth{%
\unhbox\onelinebox
}%
\fi
\else
\hbox to \myboxwidth{%
$\Pi_{17,63}$ spans $L_{16.13}$%
\hfil}%
\hbox to \myboxwidth{%
$36322322222322322$%
\hfil}%
\box\matricesbox
\fi
}%
\hfill\discretionary{}{}{}%
\setbox\matricesbox=\hbox{%
{$\left[\!\llap{\phantom{%
\begingroup \smaller\smaller\smaller
\endgroup%
}}\!\right]$}%
}%
\ifdim\wd\matricesbox>\halfwidth\myboxwidth=\hsize\else\myboxwidth=\halfwidth\fi
\vbox{%
\ifdim\myboxwidth=\hsize
\setbox\onelinebox=\hbox{%
\vbox{\hbox{%
$\Pi_{17,64}$ spans $L_{16.13}$%
}\hbox{%
$36322322223222322$%
}%
}%
\hfill\copy\matricesbox
}%
\ifdim\wd\onelinebox>\myboxwidth
\hbox to \myboxwidth{%
$\Pi_{17,64}$ spans $L_{16.13}$%
\hfil
$36322322223222322$%
}%
\box\matricesbox
\else
\hbox to \myboxwidth{%
\unhbox\onelinebox
}%
\fi
\else
\hbox to \myboxwidth{%
$\Pi_{17,64}$ spans $L_{16.13}$%
\hfil}%
\hbox to \myboxwidth{%
$36322322223222322$%
\hfil}%
\box\matricesbox
\fi
}%
\hfill\discretionary{}{}{}%
\setbox\matricesbox=\hbox{%
{$\left[\!\llap{\phantom{%
\begingroup \smaller\smaller\smaller
\endgroup%
}}\!\right]$}%
}%
\ifdim\wd\matricesbox>\halfwidth\myboxwidth=\hsize\else\myboxwidth=\halfwidth\fi
\vbox{%
\ifdim\myboxwidth=\hsize
\setbox\onelinebox=\hbox{%
\vbox{\hbox{%
$\Pi_{17,65}$ spans $L_{16.13}$%
}\hbox{%
$36322322223223222$%
}%
}%
\hfill\copy\matricesbox
}%
\ifdim\wd\onelinebox>\myboxwidth
\hbox to \myboxwidth{%
$\Pi_{17,65}$ spans $L_{16.13}$%
\hfil
$36322322223223222$%
}%
\box\matricesbox
\else
\hbox to \myboxwidth{%
\unhbox\onelinebox
}%
\fi
\else
\hbox to \myboxwidth{%
$\Pi_{17,65}$ spans $L_{16.13}$%
\hfil}%
\hbox to \myboxwidth{%
$36322322223223222$%
\hfil}%
\box\matricesbox
\fi
}%
\hfill\discretionary{}{}{}%
\setbox\matricesbox=\hbox{%
{$\left[\!\llap{\phantom{%
\begingroup \smaller\smaller\smaller
\endgroup%
}}\!\right]$}%
}%
\ifdim\wd\matricesbox>\halfwidth\myboxwidth=\hsize\else\myboxwidth=\halfwidth\fi
\vbox{%
\ifdim\myboxwidth=\hsize
\setbox\onelinebox=\hbox{%
\vbox{\hbox{%
$\Pi_{17,66}$ spans $L_{16.13}$%
}\hbox{%
$36322322232222322$%
}%
}%
\hfill\copy\matricesbox
}%
\ifdim\wd\onelinebox>\myboxwidth
\hbox to \myboxwidth{%
$\Pi_{17,66}$ spans $L_{16.13}$%
\hfil
$36322322232222322$%
}%
\box\matricesbox
\else
\hbox to \myboxwidth{%
\unhbox\onelinebox
}%
\fi
\else
\hbox to \myboxwidth{%
$\Pi_{17,66}$ spans $L_{16.13}$%
\hfil}%
\hbox to \myboxwidth{%
$36322322232222322$%
\hfil}%
\box\matricesbox
\fi
}%
\hfill\discretionary{}{}{}%
\setbox\matricesbox=\hbox{%
{$\left[\!\llap{\phantom{%
\begingroup \smaller\smaller\smaller
\endgroup%
}}\!\right]$}%
}%
\ifdim\wd\matricesbox>\halfwidth\myboxwidth=\hsize\else\myboxwidth=\halfwidth\fi
\vbox{%
\ifdim\myboxwidth=\hsize
\setbox\onelinebox=\hbox{%
\vbox{\hbox{%
$\Pi_{17,67}$ spans $L_{16.13}$%
}\hbox{%
$36322322232223222$%
}%
}%
\hfill\copy\matricesbox
}%
\ifdim\wd\onelinebox>\myboxwidth
\hbox to \myboxwidth{%
$\Pi_{17,67}$ spans $L_{16.13}$%
\hfil
$36322322232223222$%
}%
\box\matricesbox
\else
\hbox to \myboxwidth{%
\unhbox\onelinebox
}%
\fi
\else
\hbox to \myboxwidth{%
$\Pi_{17,67}$ spans $L_{16.13}$%
\hfil}%
\hbox to \myboxwidth{%
$36322322232223222$%
\hfil}%
\box\matricesbox
\fi
}%
\hfill\discretionary{}{}{}%
\setbox\matricesbox=\hbox{%
{$\left[\!\llap{\phantom{%
\begingroup \smaller\smaller\smaller
\endgroup%
}}\!\right]$}%
}%
\ifdim\wd\matricesbox>\halfwidth\myboxwidth=\hsize\else\myboxwidth=\halfwidth\fi
\vbox{%
\ifdim\myboxwidth=\hsize
\setbox\onelinebox=\hbox{%
\vbox{\hbox{%
$\Pi_{17,68}$ spans $L_{16.13}$%
}\hbox{%
$36322322232232222$%
}%
}%
\hfill\copy\matricesbox
}%
\ifdim\wd\onelinebox>\myboxwidth
\hbox to \myboxwidth{%
$\Pi_{17,68}$ spans $L_{16.13}$%
\hfil
$36322322232232222$%
}%
\box\matricesbox
\else
\hbox to \myboxwidth{%
\unhbox\onelinebox
}%
\fi
\else
\hbox to \myboxwidth{%
$\Pi_{17,68}$ spans $L_{16.13}$%
\hfil}%
\hbox to \myboxwidth{%
$36322322232232222$%
\hfil}%
\box\matricesbox
\fi
}%
\hfill\discretionary{}{}{}%
\setbox\matricesbox=\hbox{%
{$\left[\!\llap{\phantom{%
\begingroup \smaller\smaller\smaller
\endgroup%
}}\!\right]$}%
}%
\ifdim\wd\matricesbox>\halfwidth\myboxwidth=\hsize\else\myboxwidth=\halfwidth\fi
\vbox{%
\ifdim\myboxwidth=\hsize
\setbox\onelinebox=\hbox{%
\vbox{\hbox{%
$\Pi_{17,69}$ spans $L_{16.13}$%
}\hbox{%
$36322322322222322$%
}%
}%
\hfill\copy\matricesbox
}%
\ifdim\wd\onelinebox>\myboxwidth
\hbox to \myboxwidth{%
$\Pi_{17,69}$ spans $L_{16.13}$%
\hfil
$36322322322222322$%
}%
\box\matricesbox
\else
\hbox to \myboxwidth{%
\unhbox\onelinebox
}%
\fi
\else
\hbox to \myboxwidth{%
$\Pi_{17,69}$ spans $L_{16.13}$%
\hfil}%
\hbox to \myboxwidth{%
$36322322322222322$%
\hfil}%
\box\matricesbox
\fi
}%
\hfill\discretionary{}{}{}%
\setbox\matricesbox=\hbox{%
{$\left[\!\llap{\phantom{%
\begingroup \smaller\smaller\smaller
\endgroup%
}}\!\right]$}%
}%
\ifdim\wd\matricesbox>\halfwidth\myboxwidth=\hsize\else\myboxwidth=\halfwidth\fi
\vbox{%
\ifdim\myboxwidth=\hsize
\setbox\onelinebox=\hbox{%
\vbox{\hbox{%
$\Pi_{17,70}$ spans $L_{16.13}$%
}\hbox{%
$36322322322223222$%
}%
}%
\hfill\copy\matricesbox
}%
\ifdim\wd\onelinebox>\myboxwidth
\hbox to \myboxwidth{%
$\Pi_{17,70}$ spans $L_{16.13}$%
\hfil
$36322322322223222$%
}%
\box\matricesbox
\else
\hbox to \myboxwidth{%
\unhbox\onelinebox
}%
\fi
\else
\hbox to \myboxwidth{%
$\Pi_{17,70}$ spans $L_{16.13}$%
\hfil}%
\hbox to \myboxwidth{%
$36322322322223222$%
\hfil}%
\box\matricesbox
\fi
}%
\hfill\discretionary{}{}{}%
\setbox\matricesbox=\hbox{%
{$\left[\!\llap{\phantom{%
\begingroup \smaller\smaller\smaller
\endgroup%
}}\!\right]$}%
}%
\ifdim\wd\matricesbox>\halfwidth\myboxwidth=\hsize\else\myboxwidth=\halfwidth\fi
\vbox{%
\ifdim\myboxwidth=\hsize
\setbox\onelinebox=\hbox{%
\vbox{\hbox{%
$\Pi_{17,71}$ spans $L_{16.13}$%
}\hbox{%
$36322322322232222$%
}%
}%
\hfill\copy\matricesbox
}%
\ifdim\wd\onelinebox>\myboxwidth
\hbox to \myboxwidth{%
$\Pi_{17,71}$ spans $L_{16.13}$%
\hfil
$36322322322232222$%
}%
\box\matricesbox
\else
\hbox to \myboxwidth{%
\unhbox\onelinebox
}%
\fi
\else
\hbox to \myboxwidth{%
$\Pi_{17,71}$ spans $L_{16.13}$%
\hfil}%
\hbox to \myboxwidth{%
$36322322322232222$%
\hfil}%
\box\matricesbox
\fi
}%
\hfill\discretionary{}{}{}%
\setbox\matricesbox=\hbox{%
{$\left[\!\llap{\phantom{%
\begingroup \smaller\smaller\smaller
\endgroup%
}}\!\right]$}%
}%
\ifdim\wd\matricesbox>\halfwidth\myboxwidth=\hsize\else\myboxwidth=\halfwidth\fi
\vbox{%
\ifdim\myboxwidth=\hsize
\setbox\onelinebox=\hbox{%
\vbox{\hbox{%
$\Pi_{17,72}$ spans $L_{16.13}$%
}\hbox{%
$36322322322322222$%
}%
}%
\hfill\copy\matricesbox
}%
\ifdim\wd\onelinebox>\myboxwidth
\hbox to \myboxwidth{%
$\Pi_{17,72}$ spans $L_{16.13}$%
\hfil
$36322322322322222$%
}%
\box\matricesbox
\else
\hbox to \myboxwidth{%
\unhbox\onelinebox
}%
\fi
\else
\hbox to \myboxwidth{%
$\Pi_{17,72}$ spans $L_{16.13}$%
\hfil}%
\hbox to \myboxwidth{%
$36322322322322222$%
\hfil}%
\box\matricesbox
\fi
}%
\hfill\discretionary{}{}{}%
\setbox\matricesbox=\hbox{%
{$\left[\!\llap{\phantom{%
\begingroup \smaller\smaller\smaller
\endgroup%
}}\!\right]$}%
}%
\ifdim\wd\matricesbox>\halfwidth\myboxwidth=\hsize\else\myboxwidth=\halfwidth\fi
\vbox{%
\ifdim\myboxwidth=\hsize
\setbox\onelinebox=\hbox{%
\vbox{\hbox{%
$\Pi_{17,73}$ spans $L_{16.13}$%
}\hbox{%
$36322363222223222$%
}%
}%
\hfill\copy\matricesbox
}%
\ifdim\wd\onelinebox>\myboxwidth
\hbox to \myboxwidth{%
$\Pi_{17,73}$ spans $L_{16.13}$%
\hfil
$36322363222223222$%
}%
\box\matricesbox
\else
\hbox to \myboxwidth{%
\unhbox\onelinebox
}%
\fi
\else
\hbox to \myboxwidth{%
$\Pi_{17,73}$ spans $L_{16.13}$%
\hfil}%
\hbox to \myboxwidth{%
$36322363222223222$%
\hfil}%
\box\matricesbox
\fi
}%
\hfill\discretionary{}{}{}%
\setbox\matricesbox=\hbox{%
{$\left[\!\llap{\phantom{%
\begingroup \smaller\smaller\smaller
\endgroup%
}}\!\right]$}%
}%
\ifdim\wd\matricesbox>\halfwidth\myboxwidth=\hsize\else\myboxwidth=\halfwidth\fi
\vbox{%
\ifdim\myboxwidth=\hsize
\setbox\onelinebox=\hbox{%
\vbox{\hbox{%
$\Pi_{17,74}$ spans $L_{16.13}$%
}\hbox{%
$36363222223222322$%
}%
}%
\hfill\copy\matricesbox
}%
\ifdim\wd\onelinebox>\myboxwidth
\hbox to \myboxwidth{%
$\Pi_{17,74}$ spans $L_{16.13}$%
\hfil
$36363222223222322$%
}%
\box\matricesbox
\else
\hbox to \myboxwidth{%
\unhbox\onelinebox
}%
\fi
\else
\hbox to \myboxwidth{%
$\Pi_{17,74}$ spans $L_{16.13}$%
\hfil}%
\hbox to \myboxwidth{%
$36363222223222322$%
\hfil}%
\box\matricesbox
\fi
}%
\hfill\discretionary{}{}{}%
\setbox\matricesbox=\hbox{%
{$\left[\!\llap{\phantom{%
\begingroup \smaller\smaller\smaller
\endgroup%
}}\!\right]$}%
}%
\ifdim\wd\matricesbox>\halfwidth\myboxwidth=\hsize\else\myboxwidth=\halfwidth\fi
\vbox{%
\ifdim\myboxwidth=\hsize
\setbox\onelinebox=\hbox{%
\vbox{\hbox{%
$\Pi_{17,75}$ spans $L_{16.13}$%
}\hbox{%
$36363222223223222$%
}%
}%
\hfill\copy\matricesbox
}%
\ifdim\wd\onelinebox>\myboxwidth
\hbox to \myboxwidth{%
$\Pi_{17,75}$ spans $L_{16.13}$%
\hfil
$36363222223223222$%
}%
\box\matricesbox
\else
\hbox to \myboxwidth{%
\unhbox\onelinebox
}%
\fi
\else
\hbox to \myboxwidth{%
$\Pi_{17,75}$ spans $L_{16.13}$%
\hfil}%
\hbox to \myboxwidth{%
$36363222223223222$%
\hfil}%
\box\matricesbox
\fi
}%
\hfill\discretionary{}{}{}%
\setbox\matricesbox=\hbox{%
{$\left[\!\llap{\phantom{%
\begingroup \smaller\smaller\smaller
\endgroup%
}}\!\right]$}%
}%
\ifdim\wd\matricesbox>\halfwidth\myboxwidth=\hsize\else\myboxwidth=\halfwidth\fi
\vbox{%
\ifdim\myboxwidth=\hsize
\setbox\onelinebox=\hbox{%
\vbox{\hbox{%
$\Pi_{17,76}$ spans $L_{16.13}$%
}\hbox{%
$36363222232222322$%
}%
}%
\hfill\copy\matricesbox
}%
\ifdim\wd\onelinebox>\myboxwidth
\hbox to \myboxwidth{%
$\Pi_{17,76}$ spans $L_{16.13}$%
\hfil
$36363222232222322$%
}%
\box\matricesbox
\else
\hbox to \myboxwidth{%
\unhbox\onelinebox
}%
\fi
\else
\hbox to \myboxwidth{%
$\Pi_{17,76}$ spans $L_{16.13}$%
\hfil}%
\hbox to \myboxwidth{%
$36363222232222322$%
\hfil}%
\box\matricesbox
\fi
}%
\hfill\discretionary{}{}{}%
\setbox\matricesbox=\hbox{%
{$\left[\!\llap{\phantom{%
\begingroup \smaller\smaller\smaller
\endgroup%
}}\!\right]$}%
}%
\ifdim\wd\matricesbox>\halfwidth\myboxwidth=\hsize\else\myboxwidth=\halfwidth\fi
\vbox{%
\ifdim\myboxwidth=\hsize
\setbox\onelinebox=\hbox{%
\vbox{\hbox{%
$\Pi_{17,77}$ spans $L_{16.13}$%
}\hbox{%
$36363222232223222$%
}%
}%
\hfill\copy\matricesbox
}%
\ifdim\wd\onelinebox>\myboxwidth
\hbox to \myboxwidth{%
$\Pi_{17,77}$ spans $L_{16.13}$%
\hfil
$36363222232223222$%
}%
\box\matricesbox
\else
\hbox to \myboxwidth{%
\unhbox\onelinebox
}%
\fi
\else
\hbox to \myboxwidth{%
$\Pi_{17,77}$ spans $L_{16.13}$%
\hfil}%
\hbox to \myboxwidth{%
$36363222232223222$%
\hfil}%
\box\matricesbox
\fi
}%
\hfill\discretionary{}{}{}%
\setbox\matricesbox=\hbox{%
{$\left[\!\llap{\phantom{%
\begingroup \smaller\smaller\smaller
\endgroup%
}}\!\right]$}%
}%
\ifdim\wd\matricesbox>\halfwidth\myboxwidth=\hsize\else\myboxwidth=\halfwidth\fi
\vbox{%
\ifdim\myboxwidth=\hsize
\setbox\onelinebox=\hbox{%
\vbox{\hbox{%
$\Pi_{17,78}$ spans $L_{16.13}$%
}\hbox{%
$36363222322222322$%
}%
}%
\hfill\copy\matricesbox
}%
\ifdim\wd\onelinebox>\myboxwidth
\hbox to \myboxwidth{%
$\Pi_{17,78}$ spans $L_{16.13}$%
\hfil
$36363222322222322$%
}%
\box\matricesbox
\else
\hbox to \myboxwidth{%
\unhbox\onelinebox
}%
\fi
\else
\hbox to \myboxwidth{%
$\Pi_{17,78}$ spans $L_{16.13}$%
\hfil}%
\hbox to \myboxwidth{%
$36363222322222322$%
\hfil}%
\box\matricesbox
\fi
}%
\hfill\discretionary{}{}{}%
\setbox\matricesbox=\hbox{%
{$\left[\!\llap{\phantom{%
\begingroup \smaller\smaller\smaller
\endgroup%
}}\!\right]$}%
}%
\ifdim\wd\matricesbox>\halfwidth\myboxwidth=\hsize\else\myboxwidth=\halfwidth\fi
\vbox{%
\ifdim\myboxwidth=\hsize
\setbox\onelinebox=\hbox{%
\vbox{\hbox{%
$\Pi_{17,79}$ spans $L_{142.20}$%
}\hbox{%
$\infty422\infty224\infty4224\infty422$%
}%
}%
\hfill\copy\matricesbox
}%
\ifdim\wd\onelinebox>\myboxwidth
\hbox to \myboxwidth{%
$\Pi_{17,79}$ spans $L_{142.20}$%
\hfil
$\infty422\infty224\infty4224\infty422$%
}%
\box\matricesbox
\else
\hbox to \myboxwidth{%
\unhbox\onelinebox
}%
\fi
\else
\hbox to \myboxwidth{%
$\Pi_{17,79}$ spans $L_{142.20}$%
\hfil}%
\hbox to \myboxwidth{%
$\infty422\infty224\infty4224\infty422$%
\hfil}%
\box\matricesbox
\fi
}%
\hfill\discretionary{}{}{}%
\setbox\matricesbox=\hbox{%
{$\left[\!\llap{\phantom{%
\begingroup \smaller\smaller\smaller
\endgroup%
}}\!\right]$}%
}%
\ifdim\wd\matricesbox>\halfwidth\myboxwidth=\hsize\else\myboxwidth=\halfwidth\fi
\vbox{%
\ifdim\myboxwidth=\hsize
\setbox\onelinebox=\hbox{%
\vbox{\hbox{%
$\Pi_{17,80}$ spans $L_{142.20}$%
}\hbox{%
$\infty422\infty422\infty4224\infty422$%
}%
}%
\hfill\copy\matricesbox
}%
\ifdim\wd\onelinebox>\myboxwidth
\hbox to \myboxwidth{%
$\Pi_{17,80}$ spans $L_{142.20}$%
\hfil
$\infty422\infty422\infty4224\infty422$%
}%
\box\matricesbox
\else
\hbox to \myboxwidth{%
\unhbox\onelinebox
}%
\fi
\else
\hbox to \myboxwidth{%
$\Pi_{17,80}$ spans $L_{142.20}$%
\hfil}%
\hbox to \myboxwidth{%
$\infty422\infty422\infty4224\infty422$%
\hfil}%
\box\matricesbox
\fi
}%
\hfill\discretionary{}{}{}%
\setbox\matricesbox=\hbox{%
{$\left[\!\llap{\phantom{%
\begingroup \smaller\smaller\smaller
\endgroup%
}}\!\right]$}%
}%
\ifdim\wd\matricesbox>\halfwidth\myboxwidth=\hsize\else\myboxwidth=\halfwidth\fi
\vbox{%
\ifdim\myboxwidth=\hsize
\setbox\onelinebox=\hbox{%
\vbox{\hbox{%
$\Pi_{17,81}$ spans $L_{142.20}$%
}\hbox{%
$\infty422\infty4224\infty224\infty422$%
}%
}%
\hfill\copy\matricesbox
}%
\ifdim\wd\onelinebox>\myboxwidth
\hbox to \myboxwidth{%
$\Pi_{17,81}$ spans $L_{142.20}$%
\hfil
$\infty422\infty4224\infty224\infty422$%
}%
\box\matricesbox
\else
\hbox to \myboxwidth{%
\unhbox\onelinebox
}%
\fi
\else
\hbox to \myboxwidth{%
$\Pi_{17,81}$ spans $L_{142.20}$%
\hfil}%
\hbox to \myboxwidth{%
$\infty422\infty4224\infty224\infty422$%
\hfil}%
\box\matricesbox
\fi
}%
\hfill\discretionary{}{}{}%
\setbox\matricesbox=\hbox{%
{$\left[\!\llap{\phantom{%
\begingroup \smaller\smaller\smaller
\endgroup%
}}\!\right]$}%
}%
\ifdim\wd\matricesbox>\halfwidth\myboxwidth=\hsize\else\myboxwidth=\halfwidth\fi
\vbox{%
\ifdim\myboxwidth=\hsize
\setbox\onelinebox=\hbox{%
\vbox{\hbox{%
$\Pi_{17,82}$ spans $L_{142.20}$%
}\hbox{%
$\infty422\infty4224\infty4224\infty22$%
}%
}%
\hfill\copy\matricesbox
}%
\ifdim\wd\onelinebox>\myboxwidth
\hbox to \myboxwidth{%
$\Pi_{17,82}$ spans $L_{142.20}$%
\hfil
$\infty422\infty4224\infty4224\infty22$%
}%
\box\matricesbox
\else
\hbox to \myboxwidth{%
\unhbox\onelinebox
}%
\fi
\else
\hbox to \myboxwidth{%
$\Pi_{17,82}$ spans $L_{142.20}$%
\hfil}%
\hbox to \myboxwidth{%
$\infty422\infty4224\infty4224\infty22$%
\hfil}%
\box\matricesbox
\fi
}%
\hfill\discretionary{}{}{}%
\setbox\matricesbox=\hbox{%
{$\left[\!\llap{\phantom{%
\begingroup \smaller\smaller\smaller
\endgroup%
}}\!\right]$}%
}%
\ifdim\wd\matricesbox>\halfwidth\myboxwidth=\hsize\else\myboxwidth=\halfwidth\fi
\vbox{%
\ifdim\myboxwidth=\hsize
\setbox\onelinebox=\hbox{%
\vbox{\hbox{%
$\Pi_{17,83}$ spans $L_{142.20}$%
}\hbox{%
$\infty4224\infty22\infty4224\infty422$%
}%
}%
\hfill\copy\matricesbox
}%
\ifdim\wd\onelinebox>\myboxwidth
\hbox to \myboxwidth{%
$\Pi_{17,83}$ spans $L_{142.20}$%
\hfil
$\infty4224\infty22\infty4224\infty422$%
}%
\box\matricesbox
\else
\hbox to \myboxwidth{%
\unhbox\onelinebox
}%
\fi
\else
\hbox to \myboxwidth{%
$\Pi_{17,83}$ spans $L_{142.20}$%
\hfil}%
\hbox to \myboxwidth{%
$\infty4224\infty22\infty4224\infty422$%
\hfil}%
\box\matricesbox
\fi
}%
\hfill\discretionary{}{}{}%
\setbox\matricesbox=\hbox{%
{$\left[\!\llap{\phantom{%
\begingroup \smaller\smaller\smaller
\endgroup%
}}\!\right]$}%
}%
\ifdim\wd\matricesbox>\halfwidth\myboxwidth=\hsize\else\myboxwidth=\halfwidth\fi
\vbox{%
\ifdim\myboxwidth=\hsize
\setbox\onelinebox=\hbox{%
\vbox{\hbox{%
$\Pi_{17,84}$ spans $L_{142.20}$%
}\hbox{%
$\infty4224\infty422\infty224\infty422$%
}%
}%
\hfill\copy\matricesbox
}%
\ifdim\wd\onelinebox>\myboxwidth
\hbox to \myboxwidth{%
$\Pi_{17,84}$ spans $L_{142.20}$%
\hfil
$\infty4224\infty422\infty224\infty422$%
}%
\box\matricesbox
\else
\hbox to \myboxwidth{%
\unhbox\onelinebox
}%
\fi
\else
\hbox to \myboxwidth{%
$\Pi_{17,84}$ spans $L_{142.20}$%
\hfil}%
\hbox to \myboxwidth{%
$\infty4224\infty422\infty224\infty422$%
\hfil}%
\box\matricesbox
\fi
}%
\hfill\discretionary{}{}{}%
\setbox\matricesbox=\hbox{%
{$\left[\!\llap{\phantom{%
\begingroup \smaller\smaller\smaller
\endgroup%
}}\!\right]$}%
}%
\ifdim\wd\matricesbox>\halfwidth\myboxwidth=\hsize\else\myboxwidth=\halfwidth\fi
\vbox{%
\ifdim\myboxwidth=\hsize
\setbox\onelinebox=\hbox{%
\vbox{\hbox{%
$\Pi_{17,85}$ spans $L_{142.20}$%
}\hbox{%
$\infty4224\infty4224\infty422\infty22$%
}%
}%
\hfill\copy\matricesbox
}%
\ifdim\wd\onelinebox>\myboxwidth
\hbox to \myboxwidth{%
$\Pi_{17,85}$ spans $L_{142.20}$%
\hfil
$\infty4224\infty4224\infty422\infty22$%
}%
\box\matricesbox
\else
\hbox to \myboxwidth{%
\unhbox\onelinebox
}%
\fi
\else
\hbox to \myboxwidth{%
$\Pi_{17,85}$ spans $L_{142.20}$%
\hfil}%
\hbox to \myboxwidth{%
$\infty4224\infty4224\infty422\infty22$%
\hfil}%
\box\matricesbox
\fi
}%
\hfill\discretionary{}{}{}%

\vskip2pt\hrule\vskip2pt

\leavevmode\setbox\matricesbox=\hbox{%
{$\left[\!\llap{\phantom{%
\begingroup \smaller\smaller\smaller\begin{tabular}{@{}c@{}}%
\phantom{0}\\\phantom{0}\\\phantom{0}\\\phantom{0}
\end{tabular}\endgroup%
}}\right.$}%
\begingroup \smaller\smaller\smaller\begin{tabular}{@{}c@{}}%
-5\\\phantom{0}\\\phantom{0}\\\phantom{0}
\end{tabular}\endgroup%
\kern3pt%
\begingroup \smaller\smaller\smaller\begin{tabular}{@{}c@{}}%
\phantom{0}\\6\\\phantom{0}\\\phantom{0}
\end{tabular}\endgroup%
\kern3pt%
\begingroup \smaller\smaller\smaller\begin{tabular}{@{}c@{}}%
\phantom{0}\\\phantom{0}\\6\\\phantom{0}
\end{tabular}\endgroup%
\kern3pt%
\begingroup \smaller\smaller\smaller\begin{tabular}{@{}c@{}}%
\phantom{0}\\\phantom{0}\\\phantom{0}\\6
\end{tabular}\endgroup%
{$\left.\llap{\phantom{%
\begingroup \smaller\smaller\smaller\begin{tabular}{@{}c@{}}%
\phantom{0}\\\phantom{0}\\\phantom{0}\\\phantom{0}
\end{tabular}\endgroup%
}}\!\right]$}%
{$\left[\!\llap{\phantom{%
\begingroup \smaller\smaller\smaller\begin{tabular}{@{}c@{}}%
0\\0\\0\\0
\end{tabular}\endgroup%
}}\right.$}%
\begingroup \smaller\smaller\smaller\begin{tabular}{@{}c@{}}%
4\\3\\-1\\-2
\end{tabular}\endgroup%
\kern3pt%
\begingroup \smaller\smaller\smaller\begin{tabular}{@{}c@{}}%
3\\2\\0\\-2
\end{tabular}\endgroup%
{$\left.\llap{\phantom{%
\begingroup \smaller\smaller\smaller\begin{tabular}{@{}c@{}}%
0\\0\\0\\0
\end{tabular}\endgroup%
}}\!\right]$}%
}%
\ifdim\wd\matricesbox>\halfwidth\myboxwidth=\hsize\else\myboxwidth=\halfwidth\fi
\vbox{%
\ifdim\myboxwidth=\hsize
\setbox\onelinebox=\hbox{%
\vbox{\hbox{%
$\Pi_{18,1}$ spans $L_{251.3}$%
}\hbox{%
$\slashthree2|2\slashthree2|2\slashthree2|2\slashthree2|2\slashthree2|2\slashthree2|2\rtimes D_{12}$%
}%
}%
\hfill\copy\matricesbox
}%
\ifdim\wd\onelinebox>\myboxwidth
\hbox to \myboxwidth{%
$\Pi_{18,1}$ spans $L_{251.3}$%
\hfil
$\slashthree2|2\slashthree2|2\slashthree2|2\slashthree2|2\slashthree2|2\slashthree2|2\rtimes D_{12}$%
}%
\box\matricesbox
\else
\hbox to \myboxwidth{%
\unhbox\onelinebox
}%
\fi
\else
\hbox to \myboxwidth{%
$\Pi_{18,1}$ spans $L_{251.3}$%
\hfil}%
\hbox to \myboxwidth{%
$\slashthree2|2\slashthree2|2\slashthree2|2\slashthree2|2\slashthree2|2\slashthree2|2\rtimes D_{12}$%
\hfil}%
\box\matricesbox
\fi
}%
\hfill\discretionary{}{}{}%
\setbox\matricesbox=\hbox{%
{$\left[\!\llap{\phantom{%
\begingroup \smaller\smaller\smaller\begin{tabular}{@{}c@{}}%
\phantom{0}\\\phantom{0}\\\phantom{0}\\\phantom{0}
\end{tabular}\endgroup%
}}\right.$}%
\begingroup \smaller\smaller\smaller\begin{tabular}{@{}c@{}}%
-3\\\phantom{0}\\\phantom{0}\\\phantom{0}
\end{tabular}\endgroup%
\kern3pt%
\begingroup \smaller\smaller\smaller\begin{tabular}{@{}c@{}}%
\phantom{0}\\5\\\phantom{0}\\\phantom{0}
\end{tabular}\endgroup%
\kern3pt%
\begingroup \smaller\smaller\smaller\begin{tabular}{@{}c@{}}%
\phantom{0}\\\phantom{0}\\5\\\phantom{0}
\end{tabular}\endgroup%
\kern3pt%
\begingroup \smaller\smaller\smaller\begin{tabular}{@{}c@{}}%
\phantom{0}\\\phantom{0}\\\phantom{0}\\5
\end{tabular}\endgroup%
{$\left.\llap{\phantom{%
\begingroup \smaller\smaller\smaller\begin{tabular}{@{}c@{}}%
\phantom{0}\\\phantom{0}\\\phantom{0}\\\phantom{0}
\end{tabular}\endgroup%
}}\!\right]$}%
{$\left[\!\llap{\phantom{%
\begingroup \smaller\smaller\smaller\begin{tabular}{@{}c@{}}%
0\\0\\0\\0
\end{tabular}\endgroup%
}}\right.$}%
\begingroup \smaller\smaller\smaller\begin{tabular}{@{}c@{}}%
10\\-5\\-1\\6
\end{tabular}\endgroup%
\kern3pt%
\begingroup \smaller\smaller\smaller\begin{tabular}{@{}c@{}}%
3\\-1\\-1\\2
\end{tabular}\endgroup%
{$\left.\llap{\phantom{%
\begingroup \smaller\smaller\smaller\begin{tabular}{@{}c@{}}%
0\\0\\0\\0
\end{tabular}\endgroup%
}}\!\right]$}%
}%
\ifdim\wd\matricesbox>\halfwidth\myboxwidth=\hsize\else\myboxwidth=\halfwidth\fi
\vbox{%
\ifdim\myboxwidth=\hsize
\setbox\onelinebox=\hbox{%
\vbox{\hbox{%
$\Pi_{18,2}$ spans $L_{16.9}$%
}\hbox{%
$\slashthree2|2\slashthree2|2\slashthree2|2\slashthree2|2\slashthree2|2\slashthree2|2\rtimes D_{12}$%
}%
}%
\hfill\copy\matricesbox
}%
\ifdim\wd\onelinebox>\myboxwidth
\hbox to \myboxwidth{%
$\Pi_{18,2}$ spans $L_{16.9}$%
\hfil
$\slashthree2|2\slashthree2|2\slashthree2|2\slashthree2|2\slashthree2|2\slashthree2|2\rtimes D_{12}$%
}%
\box\matricesbox
\else
\hbox to \myboxwidth{%
\unhbox\onelinebox
}%
\fi
\else
\hbox to \myboxwidth{%
$\Pi_{18,2}$ spans $L_{16.9}$%
\hfil}%
\hbox to \myboxwidth{%
$\slashthree2|2\slashthree2|2\slashthree2|2\slashthree2|2\slashthree2|2\slashthree2|2\rtimes D_{12}$%
\hfil}%
\box\matricesbox
\fi
}%
\hfill\discretionary{}{}{}%
\setbox\matricesbox=\hbox{%
{$\left[\!\llap{\phantom{%
\begingroup \smaller\smaller\smaller\begin{tabular}{@{}c@{}}%
\phantom{0}\\\phantom{0}\\\phantom{0}\\\phantom{0}
\end{tabular}\endgroup%
}}\right.$}%
\begingroup \smaller\smaller\smaller\begin{tabular}{@{}c@{}}%
-7\\\phantom{0}\\\phantom{0}\\\phantom{0}
\end{tabular}\endgroup%
\kern3pt%
\begingroup \smaller\smaller\smaller\begin{tabular}{@{}c@{}}%
\phantom{0}\\1\\\phantom{0}\\\phantom{0}
\end{tabular}\endgroup%
\kern3pt%
\begingroup \smaller\smaller\smaller\begin{tabular}{@{}c@{}}%
\phantom{0}\\\phantom{0}\\1\\\phantom{0}
\end{tabular}\endgroup%
\kern3pt%
\begingroup \smaller\smaller\smaller\begin{tabular}{@{}c@{}}%
\phantom{0}\\\phantom{0}\\\phantom{0}\\1
\end{tabular}\endgroup%
{$\left.\llap{\phantom{%
\begingroup \smaller\smaller\smaller\begin{tabular}{@{}c@{}}%
\phantom{0}\\\phantom{0}\\\phantom{0}\\\phantom{0}
\end{tabular}\endgroup%
}}\!\right]$}%
{$\left[\!\llap{\phantom{%
\begingroup \smaller\smaller\smaller\begin{tabular}{@{}c@{}}%
0\\0\\0\\0
\end{tabular}\endgroup%
}}\right.$}%
\begingroup \smaller\smaller\smaller\begin{tabular}{@{}c@{}}%
6\\8\\-13\\5
\end{tabular}\endgroup%
\kern3pt%
\begingroup \smaller\smaller\smaller\begin{tabular}{@{}c@{}}%
1\\2\\-2\\0
\end{tabular}\endgroup%
{$\left.\llap{\phantom{%
\begingroup \smaller\smaller\smaller\begin{tabular}{@{}c@{}}%
0\\0\\0\\0
\end{tabular}\endgroup%
}}\!\right]$}%
}%
\ifdim\wd\matricesbox>\halfwidth\myboxwidth=\hsize\else\myboxwidth=\halfwidth\fi
\vbox{%
\ifdim\myboxwidth=\hsize
\setbox\onelinebox=\hbox{%
\vbox{\hbox{%
$\Pi_{18,3}$ spans $L_{22.2}$%
}\hbox{%
$\slashthree2|2\slashthree2|2\slashthree2|2\slashthree2|2\slashthree2|2\slashthree2|2\rtimes D_{12}$%
}%
}%
\hfill\copy\matricesbox
}%
\ifdim\wd\onelinebox>\myboxwidth
\hbox to \myboxwidth{%
$\Pi_{18,3}$ spans $L_{22.2}$%
\hfil
$\slashthree2|2\slashthree2|2\slashthree2|2\slashthree2|2\slashthree2|2\slashthree2|2\rtimes D_{12}$%
}%
\box\matricesbox
\else
\hbox to \myboxwidth{%
\unhbox\onelinebox
}%
\fi
\else
\hbox to \myboxwidth{%
$\Pi_{18,3}$ spans $L_{22.2}$%
\hfil}%
\hbox to \myboxwidth{%
$\slashthree2|2\slashthree2|2\slashthree2|2\slashthree2|2\slashthree2|2\slashthree2|2\rtimes D_{12}$%
\hfil}%
\box\matricesbox
\fi
}%
\hfill\discretionary{}{}{}%
\setbox\matricesbox=\hbox{%
{$\left[\!\llap{\phantom{%
\begingroup \smaller\smaller\smaller\begin{tabular}{@{}c@{}}%
\phantom{0}\\\phantom{0}\\\phantom{0}\\\phantom{0}
\end{tabular}\endgroup%
}}\right.$}%
\begingroup \smaller\smaller\smaller\begin{tabular}{@{}c@{}}%
-5\\\phantom{0}\\\phantom{0}\\\phantom{0}
\end{tabular}\endgroup%
\kern3pt%
\begingroup \smaller\smaller\smaller\begin{tabular}{@{}c@{}}%
\phantom{0}\\3\\\phantom{0}\\\phantom{0}
\end{tabular}\endgroup%
\kern3pt%
\begingroup \smaller\smaller\smaller\begin{tabular}{@{}c@{}}%
\phantom{0}\\\phantom{0}\\3\\\phantom{0}
\end{tabular}\endgroup%
\kern3pt%
\begingroup \smaller\smaller\smaller\begin{tabular}{@{}c@{}}%
\phantom{0}\\\phantom{0}\\\phantom{0}\\3
\end{tabular}\endgroup%
{$\left.\llap{\phantom{%
\begingroup \smaller\smaller\smaller\begin{tabular}{@{}c@{}}%
\phantom{0}\\\phantom{0}\\\phantom{0}\\\phantom{0}
\end{tabular}\endgroup%
}}\!\right]$}%
{$\left[\!\llap{\phantom{%
\begingroup \smaller\smaller\smaller\begin{tabular}{@{}c@{}}%
0\\0\\0\\0
\end{tabular}\endgroup%
}}\right.$}%
\begingroup \smaller\smaller\smaller\begin{tabular}{@{}c@{}}%
18\\11\\-19\\8
\end{tabular}\endgroup%
\kern3pt%
\begingroup \smaller\smaller\smaller\begin{tabular}{@{}c@{}}%
1\\1\\-1\\0
\end{tabular}\endgroup%
{$\left.\llap{\phantom{%
\begingroup \smaller\smaller\smaller\begin{tabular}{@{}c@{}}%
0\\0\\0\\0
\end{tabular}\endgroup%
}}\!\right]$}%
}%
\ifdim\wd\matricesbox>\halfwidth\myboxwidth=\hsize\else\myboxwidth=\halfwidth\fi
\vbox{%
\ifdim\myboxwidth=\hsize
\setbox\onelinebox=\hbox{%
\vbox{\hbox{%
$\Pi_{18,4}$ spans $L_{16.7}$%
}\hbox{%
$\slashthree2|2\slashthree2|2\slashthree2|2\slashthree2|2\slashthree2|2\slashthree2|2\rtimes D_{12}$%
}%
}%
\hfill\copy\matricesbox
}%
\ifdim\wd\onelinebox>\myboxwidth
\hbox to \myboxwidth{%
$\Pi_{18,4}$ spans $L_{16.7}$%
\hfil
$\slashthree2|2\slashthree2|2\slashthree2|2\slashthree2|2\slashthree2|2\slashthree2|2\rtimes D_{12}$%
}%
\box\matricesbox
\else
\hbox to \myboxwidth{%
\unhbox\onelinebox
}%
\fi
\else
\hbox to \myboxwidth{%
$\Pi_{18,4}$ spans $L_{16.7}$%
\hfil}%
\hbox to \myboxwidth{%
$\slashthree2|2\slashthree2|2\slashthree2|2\slashthree2|2\slashthree2|2\slashthree2|2\rtimes D_{12}$%
\hfil}%
\box\matricesbox
\fi
}%
\hfill\discretionary{}{}{}%
\setbox\matricesbox=\hbox{%
{$\left[\!\llap{\phantom{%
\begingroup \smaller\smaller\smaller\begin{tabular}{@{}c@{}}%
\phantom{0}\\\phantom{0}\\\phantom{0}
\end{tabular}\endgroup%
}}\right.$}%
\begingroup \smaller\smaller\smaller\begin{tabular}{@{}c@{}}%
-1\\\phantom{0}\\\phantom{0}
\end{tabular}\endgroup%
\kern3pt%
\begingroup \smaller\smaller\smaller\begin{tabular}{@{}c@{}}%
\phantom{0}\\45/2\\\phantom{0}
\end{tabular}\endgroup%
\kern3pt%
\begingroup \smaller\smaller\smaller\begin{tabular}{@{}c@{}}%
\phantom{0}\\\phantom{0}\\15/2
\end{tabular}\endgroup%
{$\left.\llap{\phantom{%
\begingroup \smaller\smaller\smaller\begin{tabular}{@{}c@{}}%
\phantom{0}\\\phantom{0}\\\phantom{0}
\end{tabular}\endgroup%
}}\!\right]$}%
{$\left[\!\llap{\phantom{%
\begingroup \smaller\smaller\smaller\begin{tabular}{@{}c@{}}%
0\\0\\0
\end{tabular}\endgroup%
}}\right.$}%
\begingroup \smaller\smaller\smaller\begin{tabular}{@{}c@{}}%
90\\19\\3
\end{tabular}\endgroup%
\kern3pt%
\begingroup \smaller\smaller\smaller\begin{tabular}{@{}c@{}}%
30\\6\\4
\end{tabular}\endgroup%
\kern3pt%
\begingroup \smaller\smaller\smaller\begin{tabular}{@{}c@{}}%
30\\5\\7
\end{tabular}\endgroup%
\kern3pt%
\begingroup \smaller\smaller\smaller\begin{tabular}{@{}c@{}}%
9\\1\\3
\end{tabular}\endgroup%
\kern3pt%
\begingroup \smaller\smaller\smaller\begin{tabular}{@{}c@{}}%
5\\0\\2
\end{tabular}\endgroup%
{$\left.\llap{\phantom{%
\begingroup \smaller\smaller\smaller\begin{tabular}{@{}c@{}}%
0\\0\\0
\end{tabular}\endgroup%
}}\!\right]$}%
}%
\ifdim\wd\matricesbox>\halfwidth\myboxwidth=\hsize\else\myboxwidth=\halfwidth\fi
\vbox{%
\ifdim\myboxwidth=\hsize
\setbox\onelinebox=\hbox{%
\vbox{\hbox{%
$\Pi_{18,5}$ spans $L_{16.13}$%
}\hbox{%
$\slashthree6322|2236\slashthree6322|2236\rtimes D_{4}$%
}%
}%
\hfill\copy\matricesbox
}%
\ifdim\wd\onelinebox>\myboxwidth
\hbox to \myboxwidth{%
$\Pi_{18,5}$ spans $L_{16.13}$%
\hfil
$\slashthree6322|2236\slashthree6322|2236\rtimes D_{4}$%
}%
\box\matricesbox
\else
\hbox to \myboxwidth{%
\unhbox\onelinebox
}%
\fi
\else
\hbox to \myboxwidth{%
$\Pi_{18,5}$ spans $L_{16.13}$%
\hfil}%
\hbox to \myboxwidth{%
$\slashthree6322|2236\slashthree6322|2236\rtimes D_{4}$%
\hfil}%
\box\matricesbox
\fi
}%
\hfill\discretionary{}{}{}%
\setbox\matricesbox=\hbox{%
{$\left[\!\llap{\phantom{%
\begingroup \smaller\smaller\smaller\begin{tabular}{@{}c@{}}%
\phantom{0}\\\phantom{0}\\\phantom{0}
\end{tabular}\endgroup%
}}\right.$}%
\begingroup \smaller\smaller\smaller\begin{tabular}{@{}c@{}}%
-1\\\phantom{0}\\\phantom{0}
\end{tabular}\endgroup%
\kern3pt%
\begingroup \smaller\smaller\smaller\begin{tabular}{@{}c@{}}%
\phantom{0}\\15/2\\\phantom{0}
\end{tabular}\endgroup%
\kern3pt%
\begingroup \smaller\smaller\smaller\begin{tabular}{@{}c@{}}%
\phantom{0}\\\phantom{0}\\45/2
\end{tabular}\endgroup%
{$\left.\llap{\phantom{%
\begingroup \smaller\smaller\smaller\begin{tabular}{@{}c@{}}%
\phantom{0}\\\phantom{0}\\\phantom{0}
\end{tabular}\endgroup%
}}\!\right]$}%
{$\left[\!\llap{\phantom{%
\begingroup \smaller\smaller\smaller\begin{tabular}{@{}c@{}}%
0\\0\\0
\end{tabular}\endgroup%
}}\right.$}%
\begingroup \smaller\smaller\smaller\begin{tabular}{@{}c@{}}%
30\\-11\\-1
\end{tabular}\endgroup%
\kern3pt%
\begingroup \smaller\smaller\smaller\begin{tabular}{@{}c@{}}%
90\\-30\\-8
\end{tabular}\endgroup%
\kern3pt%
\begingroup \smaller\smaller\smaller\begin{tabular}{@{}c@{}}%
90\\-27\\-11
\end{tabular}\endgroup%
\kern3pt%
\begingroup \smaller\smaller\smaller\begin{tabular}{@{}c@{}}%
5\\-1\\-1
\end{tabular}\endgroup%
\kern3pt%
\begingroup \smaller\smaller\smaller\begin{tabular}{@{}c@{}}%
9\\0\\-2
\end{tabular}\endgroup%
{$\left.\llap{\phantom{%
\begingroup \smaller\smaller\smaller\begin{tabular}{@{}c@{}}%
0\\0\\0
\end{tabular}\endgroup%
}}\!\right]$}%
}%
\ifdim\wd\matricesbox>\halfwidth\myboxwidth=\hsize\else\myboxwidth=\halfwidth\fi
\vbox{%
\ifdim\myboxwidth=\hsize
\setbox\onelinebox=\hbox{%
\vbox{\hbox{%
$\Pi_{18,6}$ spans $L_{16.13}$%
}\hbox{%
$36\slashthree6322|2236\slashthree6322|22\rtimes D_{4}$%
}%
}%
\hfill\copy\matricesbox
}%
\ifdim\wd\onelinebox>\myboxwidth
\hbox to \myboxwidth{%
$\Pi_{18,6}$ spans $L_{16.13}$%
\hfil
$36\slashthree6322|2236\slashthree6322|22\rtimes D_{4}$%
}%
\box\matricesbox
\else
\hbox to \myboxwidth{%
\unhbox\onelinebox
}%
\fi
\else
\hbox to \myboxwidth{%
$\Pi_{18,6}$ spans $L_{16.13}$%
\hfil}%
\hbox to \myboxwidth{%
$36\slashthree6322|2236\slashthree6322|22\rtimes D_{4}$%
\hfil}%
\box\matricesbox
\fi
}%
\hfill\discretionary{}{}{}%
\setbox\matricesbox=\hbox{%
{$\left[\!\llap{\phantom{%
\begingroup \smaller\smaller\smaller\begin{tabular}{@{}c@{}}%
\phantom{0}\\\phantom{0}\\\phantom{0}\\\phantom{0}
\end{tabular}\endgroup%
}}\right.$}%
\begingroup \smaller\smaller\smaller\begin{tabular}{@{}c@{}}%
-1\\\phantom{0}\\\phantom{0}\\\phantom{0}
\end{tabular}\endgroup%
\kern3pt%
\begingroup \smaller\smaller\smaller\begin{tabular}{@{}c@{}}%
\phantom{0}\\15\\\phantom{0}\\\phantom{0}
\end{tabular}\endgroup%
\kern3pt%
\begingroup \smaller\smaller\smaller\begin{tabular}{@{}c@{}}%
\phantom{0}\\\phantom{0}\\15\\\phantom{0}
\end{tabular}\endgroup%
\kern3pt%
\begingroup \smaller\smaller\smaller\begin{tabular}{@{}c@{}}%
\phantom{0}\\\phantom{0}\\\phantom{0}\\15
\end{tabular}\endgroup%
{$\left.\llap{\phantom{%
\begingroup \smaller\smaller\smaller\begin{tabular}{@{}c@{}}%
\phantom{0}\\\phantom{0}\\\phantom{0}\\\phantom{0}
\end{tabular}\endgroup%
}}\!\right]$}%
{$\left[\!\llap{\phantom{%
\begingroup \smaller\smaller\smaller\begin{tabular}{@{}c@{}}%
0\\0\\0\\0
\end{tabular}\endgroup%
}}\right.$}%
\begingroup \smaller\smaller\smaller\begin{tabular}{@{}c@{}}%
90\\19\\-11\\-8
\end{tabular}\endgroup%
\kern3pt%
\begingroup \smaller\smaller\smaller\begin{tabular}{@{}c@{}}%
90\\19\\-8\\-11
\end{tabular}\endgroup%
\kern3pt%
\begingroup \smaller\smaller\smaller\begin{tabular}{@{}c@{}}%
30\\6\\-1\\-5
\end{tabular}\endgroup%
\kern3pt%
\begingroup \smaller\smaller\smaller\begin{tabular}{@{}c@{}}%
30\\5\\1\\-6
\end{tabular}\endgroup%
\kern3pt%
\begingroup \smaller\smaller\smaller\begin{tabular}{@{}c@{}}%
9\\1\\1\\-2
\end{tabular}\endgroup%
\kern3pt%
\begingroup \smaller\smaller\smaller\begin{tabular}{@{}c@{}}%
5\\0\\1\\-1
\end{tabular}\endgroup%
{$\left.\llap{\phantom{%
\begingroup \smaller\smaller\smaller\begin{tabular}{@{}c@{}}%
0\\0\\0\\0
\end{tabular}\endgroup%
}}\!\right]$}%
}%
\ifdim\wd\matricesbox>\halfwidth\myboxwidth=\hsize\else\myboxwidth=\halfwidth\fi
\vbox{%
\ifdim\myboxwidth=\hsize
\setbox\onelinebox=\hbox{%
\vbox{\hbox{%
$\Pi_{18,7}$ spans $L_{16.13}$%
}\hbox{%
$363222363222363222\rtimes C_{3}$%
}%
}%
\hfill\copy\matricesbox
}%
\ifdim\wd\onelinebox>\myboxwidth
\hbox to \myboxwidth{%
$\Pi_{18,7}$ spans $L_{16.13}$%
\hfil
$363222363222363222\rtimes C_{3}$%
}%
\box\matricesbox
\else
\hbox to \myboxwidth{%
\unhbox\onelinebox
}%
\fi
\else
\hbox to \myboxwidth{%
$\Pi_{18,7}$ spans $L_{16.13}$%
\hfil}%
\hbox to \myboxwidth{%
$363222363222363222\rtimes C_{3}$%
\hfil}%
\box\matricesbox
\fi
}%
\hfill\discretionary{}{}{}%
\setbox\matricesbox=\hbox{%
{$\left[\!\llap{\phantom{%
\begingroup \smaller\smaller\smaller\begin{tabular}{@{}c@{}}%
\phantom{0}\\\phantom{0}\\\phantom{0}
\end{tabular}\endgroup%
}}\right.$}%
\begingroup \smaller\smaller\smaller\begin{tabular}{@{}c@{}}%
-1\\\phantom{0}\\\phantom{0}
\end{tabular}\endgroup%
\kern3pt%
\begingroup \smaller\smaller\smaller\begin{tabular}{@{}c@{}}%
\phantom{0}\\15/2\\\phantom{0}
\end{tabular}\endgroup%
\kern3pt%
\begingroup \smaller\smaller\smaller\begin{tabular}{@{}c@{}}%
\phantom{0}\\\phantom{0}\\45/2
\end{tabular}\endgroup%
{$\left.\llap{\phantom{%
\begingroup \smaller\smaller\smaller\begin{tabular}{@{}c@{}}%
\phantom{0}\\\phantom{0}\\\phantom{0}
\end{tabular}\endgroup%
}}\!\right]$}%
{$\left[\!\llap{\phantom{%
\begingroup \smaller\smaller\smaller\begin{tabular}{@{}c@{}}%
0\\0\\0
\end{tabular}\endgroup%
}}\right.$}%
\begingroup \smaller\smaller\smaller\begin{tabular}{@{}c@{}}%
5\\-2\\0
\end{tabular}\endgroup%
\kern3pt%
\begingroup \smaller\smaller\smaller\begin{tabular}{@{}c@{}}%
9\\-3\\1
\end{tabular}\endgroup%
\kern3pt%
\begingroup \smaller\smaller\smaller\begin{tabular}{@{}c@{}}%
5\\-1\\1
\end{tabular}\endgroup%
\kern3pt%
\begingroup \smaller\smaller\smaller\begin{tabular}{@{}c@{}}%
90\\-3\\19
\end{tabular}\endgroup%
\kern3pt%
\begingroup \smaller\smaller\smaller\begin{tabular}{@{}c@{}}%
90\\3\\19
\end{tabular}\endgroup%
\kern3pt%
\begingroup \smaller\smaller\smaller\begin{tabular}{@{}c@{}}%
30\\4\\6
\end{tabular}\endgroup%
\kern3pt%
\begingroup \smaller\smaller\smaller\begin{tabular}{@{}c@{}}%
30\\7\\5
\end{tabular}\endgroup%
\kern3pt%
\begingroup \smaller\smaller\smaller\begin{tabular}{@{}c@{}}%
90\\27\\11
\end{tabular}\endgroup%
\kern3pt%
\begingroup \smaller\smaller\smaller\begin{tabular}{@{}c@{}}%
90\\30\\8
\end{tabular}\endgroup%
\kern3pt%
\begingroup \smaller\smaller\smaller\begin{tabular}{@{}c@{}}%
5\\2\\0
\end{tabular}\endgroup%
{$\left.\llap{\phantom{%
\begingroup \smaller\smaller\smaller\begin{tabular}{@{}c@{}}%
0\\0\\0
\end{tabular}\endgroup%
}}\!\right]$}%
}%
\ifdim\wd\matricesbox>\halfwidth\myboxwidth=\hsize\else\myboxwidth=\halfwidth\fi
\vbox{%
\ifdim\myboxwidth=\hsize
\setbox\onelinebox=\hbox{%
\vbox{\hbox{%
$\Pi_{18,8}$ spans $L_{16.13}$%
}\hbox{%
$22|222363632|2363632\rtimes D_{2}$%
}%
}%
\hfill\copy\matricesbox
}%
\ifdim\wd\onelinebox>\myboxwidth
\hbox to \myboxwidth{%
$\Pi_{18,8}$ spans $L_{16.13}$%
\hfil
$22|222363632|2363632\rtimes D_{2}$%
}%
\box\matricesbox
\else
\hbox to \myboxwidth{%
\unhbox\onelinebox
}%
\fi
\else
\hbox to \myboxwidth{%
$\Pi_{18,8}$ spans $L_{16.13}$%
\hfil}%
\hbox to \myboxwidth{%
$22|222363632|2363632\rtimes D_{2}$%
\hfil}%
\box\matricesbox
\fi
}%
\hfill\discretionary{}{}{}%
\setbox\matricesbox=\hbox{%
{$\left[\!\llap{\phantom{%
\begingroup \smaller\smaller\smaller\begin{tabular}{@{}c@{}}%
\phantom{0}\\\phantom{0}\\\phantom{0}
\end{tabular}\endgroup%
}}\right.$}%
\begingroup \smaller\smaller\smaller\begin{tabular}{@{}c@{}}%
-1\\\phantom{0}\\\phantom{0}
\end{tabular}\endgroup%
\kern3pt%
\begingroup \smaller\smaller\smaller\begin{tabular}{@{}c@{}}%
\phantom{0}\\45/2\\\phantom{0}
\end{tabular}\endgroup%
\kern3pt%
\begingroup \smaller\smaller\smaller\begin{tabular}{@{}c@{}}%
\phantom{0}\\\phantom{0}\\15/2
\end{tabular}\endgroup%
{$\left.\llap{\phantom{%
\begingroup \smaller\smaller\smaller\begin{tabular}{@{}c@{}}%
\phantom{0}\\\phantom{0}\\\phantom{0}
\end{tabular}\endgroup%
}}\!\right]$}%
{$\left[\!\llap{\phantom{%
\begingroup \smaller\smaller\smaller\begin{tabular}{@{}c@{}}%
0\\0\\0
\end{tabular}\endgroup%
}}\right.$}%
\begingroup \smaller\smaller\smaller\begin{tabular}{@{}c@{}}%
9\\2\\0
\end{tabular}\endgroup%
\kern3pt%
\begingroup \smaller\smaller\smaller\begin{tabular}{@{}c@{}}%
5\\1\\1
\end{tabular}\endgroup%
\kern3pt%
\begingroup \smaller\smaller\smaller\begin{tabular}{@{}c@{}}%
9\\1\\3
\end{tabular}\endgroup%
\kern3pt%
\begingroup \smaller\smaller\smaller\begin{tabular}{@{}c@{}}%
30\\1\\11
\end{tabular}\endgroup%
\kern3pt%
\begingroup \smaller\smaller\smaller\begin{tabular}{@{}c@{}}%
30\\-1\\11
\end{tabular}\endgroup%
\kern3pt%
\begingroup \smaller\smaller\smaller\begin{tabular}{@{}c@{}}%
90\\-8\\30
\end{tabular}\endgroup%
\kern3pt%
\begingroup \smaller\smaller\smaller\begin{tabular}{@{}c@{}}%
90\\-11\\27
\end{tabular}\endgroup%
\kern3pt%
\begingroup \smaller\smaller\smaller\begin{tabular}{@{}c@{}}%
30\\-5\\7
\end{tabular}\endgroup%
\kern3pt%
\begingroup \smaller\smaller\smaller\begin{tabular}{@{}c@{}}%
30\\-6\\4
\end{tabular}\endgroup%
\kern3pt%
\begingroup \smaller\smaller\smaller\begin{tabular}{@{}c@{}}%
9\\-2\\0
\end{tabular}\endgroup%
{$\left.\llap{\phantom{%
\begingroup \smaller\smaller\smaller\begin{tabular}{@{}c@{}}%
0\\0\\0
\end{tabular}\endgroup%
}}\!\right]$}%
}%
\ifdim\wd\matricesbox>\halfwidth\myboxwidth=\hsize\else\myboxwidth=\halfwidth\fi
\vbox{%
\ifdim\myboxwidth=\hsize
\setbox\onelinebox=\hbox{%
\vbox{\hbox{%
$\Pi_{18,9}$ spans $L_{16.13}$%
}\hbox{%
$363222|222363632|236\rtimes D_{2}$%
}%
}%
\hfill\copy\matricesbox
}%
\ifdim\wd\onelinebox>\myboxwidth
\hbox to \myboxwidth{%
$\Pi_{18,9}$ spans $L_{16.13}$%
\hfil
$363222|222363632|236\rtimes D_{2}$%
}%
\box\matricesbox
\else
\hbox to \myboxwidth{%
\unhbox\onelinebox
}%
\fi
\else
\hbox to \myboxwidth{%
$\Pi_{18,9}$ spans $L_{16.13}$%
\hfil}%
\hbox to \myboxwidth{%
$363222|222363632|236\rtimes D_{2}$%
\hfil}%
\box\matricesbox
\fi
}%
\hfill\discretionary{}{}{}%
\setbox\matricesbox=\hbox{%
{$\left[\!\llap{\phantom{%
\begingroup \smaller\smaller\smaller\begin{tabular}{@{}c@{}}%
\phantom{0}\\\phantom{0}\\\phantom{0}
\end{tabular}\endgroup%
}}\right.$}%
\begingroup \smaller\smaller\smaller\begin{tabular}{@{}c@{}}%
-1\\\phantom{0}\\\phantom{0}
\end{tabular}\endgroup%
\kern3pt%
\begingroup \smaller\smaller\smaller\begin{tabular}{@{}c@{}}%
\phantom{0}\\15/2\\\phantom{0}
\end{tabular}\endgroup%
\kern3pt%
\begingroup \smaller\smaller\smaller\begin{tabular}{@{}c@{}}%
\phantom{0}\\\phantom{0}\\45/2
\end{tabular}\endgroup%
{$\left.\llap{\phantom{%
\begingroup \smaller\smaller\smaller\begin{tabular}{@{}c@{}}%
\phantom{0}\\\phantom{0}\\\phantom{0}
\end{tabular}\endgroup%
}}\!\right]$}%
{$\left[\!\llap{\phantom{%
\begingroup \smaller\smaller\smaller\begin{tabular}{@{}c@{}}%
0\\0\\0
\end{tabular}\endgroup%
}}\right.$}%
\begingroup \smaller\smaller\smaller\begin{tabular}{@{}c@{}}%
30\\11\\-1
\end{tabular}\endgroup%
\kern3pt%
\begingroup \smaller\smaller\smaller\begin{tabular}{@{}c@{}}%
9\\3\\-1
\end{tabular}\endgroup%
\kern3pt%
\begingroup \smaller\smaller\smaller\begin{tabular}{@{}c@{}}%
5\\1\\-1
\end{tabular}\endgroup%
\kern3pt%
\begingroup \smaller\smaller\smaller\begin{tabular}{@{}c@{}}%
9\\0\\-2
\end{tabular}\endgroup%
\kern3pt%
\begingroup \smaller\smaller\smaller\begin{tabular}{@{}c@{}}%
30\\-4\\-6
\end{tabular}\endgroup%
\kern3pt%
\begingroup \smaller\smaller\smaller\begin{tabular}{@{}c@{}}%
30\\-7\\-5
\end{tabular}\endgroup%
\kern3pt%
\begingroup \smaller\smaller\smaller\begin{tabular}{@{}c@{}}%
90\\-27\\-11
\end{tabular}\endgroup%
\kern3pt%
\begingroup \smaller\smaller\smaller\begin{tabular}{@{}c@{}}%
90\\-30\\-8
\end{tabular}\endgroup%
\kern3pt%
\begingroup \smaller\smaller\smaller\begin{tabular}{@{}c@{}}%
30\\-11\\-1
\end{tabular}\endgroup%
{$\left.\llap{\phantom{%
\begingroup \smaller\smaller\smaller\begin{tabular}{@{}c@{}}%
0\\0\\0
\end{tabular}\endgroup%
}}\!\right]$}%
}%
\ifdim\wd\matricesbox>\halfwidth\myboxwidth=\hsize\else\myboxwidth=\halfwidth\fi
\vbox{%
\ifdim\myboxwidth=\hsize
\setbox\onelinebox=\hbox{%
\vbox{\hbox{%
$\Pi_{18,10}$ spans $L_{16.13}$%
}\hbox{%
$3632222\slashthree22223636\slashthree6\rtimes D_{2}$%
}%
}%
\hfill\copy\matricesbox
}%
\ifdim\wd\onelinebox>\myboxwidth
\hbox to \myboxwidth{%
$\Pi_{18,10}$ spans $L_{16.13}$%
\hfil
$3632222\slashthree22223636\slashthree6\rtimes D_{2}$%
}%
\box\matricesbox
\else
\hbox to \myboxwidth{%
\unhbox\onelinebox
}%
\fi
\else
\hbox to \myboxwidth{%
$\Pi_{18,10}$ spans $L_{16.13}$%
\hfil}%
\hbox to \myboxwidth{%
$3632222\slashthree22223636\slashthree6\rtimes D_{2}$%
\hfil}%
\box\matricesbox
\fi
}%
\hfill\discretionary{}{}{}%
\setbox\matricesbox=\hbox{%
{$\left[\!\llap{\phantom{%
\begingroup \smaller\smaller\smaller\begin{tabular}{@{}c@{}}%
\phantom{0}\\\phantom{0}\\\phantom{0}
\end{tabular}\endgroup%
}}\right.$}%
\begingroup \smaller\smaller\smaller\begin{tabular}{@{}c@{}}%
-1\\\phantom{0}\\\phantom{0}
\end{tabular}\endgroup%
\kern3pt%
\begingroup \smaller\smaller\smaller\begin{tabular}{@{}c@{}}%
\phantom{0}\\15/2\\\phantom{0}
\end{tabular}\endgroup%
\kern3pt%
\begingroup \smaller\smaller\smaller\begin{tabular}{@{}c@{}}%
\phantom{0}\\\phantom{0}\\45/2
\end{tabular}\endgroup%
{$\left.\llap{\phantom{%
\begingroup \smaller\smaller\smaller\begin{tabular}{@{}c@{}}%
\phantom{0}\\\phantom{0}\\\phantom{0}
\end{tabular}\endgroup%
}}\!\right]$}%
{$\left[\!\llap{\phantom{%
\begingroup \smaller\smaller\smaller\begin{tabular}{@{}c@{}}%
0\\0\\0
\end{tabular}\endgroup%
}}\right.$}%
\begingroup \smaller\smaller\smaller\begin{tabular}{@{}c@{}}%
5\\-2\\0
\end{tabular}\endgroup%
\kern3pt%
\begingroup \smaller\smaller\smaller\begin{tabular}{@{}c@{}}%
9\\-3\\-1
\end{tabular}\endgroup%
\kern3pt%
\begingroup \smaller\smaller\smaller\begin{tabular}{@{}c@{}}%
30\\-7\\-5
\end{tabular}\endgroup%
\kern3pt%
\begingroup \smaller\smaller\smaller\begin{tabular}{@{}c@{}}%
30\\-4\\-6
\end{tabular}\endgroup%
\kern3pt%
\begingroup \smaller\smaller\smaller\begin{tabular}{@{}c@{}}%
90\\-3\\-19
\end{tabular}\endgroup%
\kern3pt%
\begingroup \smaller\smaller\smaller\begin{tabular}{@{}c@{}}%
90\\3\\-19
\end{tabular}\endgroup%
\kern3pt%
\begingroup \smaller\smaller\smaller\begin{tabular}{@{}c@{}}%
5\\1\\-1
\end{tabular}\endgroup%
\kern3pt%
\begingroup \smaller\smaller\smaller\begin{tabular}{@{}c@{}}%
90\\27\\-11
\end{tabular}\endgroup%
\kern3pt%
\begingroup \smaller\smaller\smaller\begin{tabular}{@{}c@{}}%
90\\30\\-8
\end{tabular}\endgroup%
\kern3pt%
\begingroup \smaller\smaller\smaller\begin{tabular}{@{}c@{}}%
5\\2\\0
\end{tabular}\endgroup%
{$\left.\llap{\phantom{%
\begingroup \smaller\smaller\smaller\begin{tabular}{@{}c@{}}%
0\\0\\0
\end{tabular}\endgroup%
}}\!\right]$}%
}%
\ifdim\wd\matricesbox>\halfwidth\myboxwidth=\hsize\else\myboxwidth=\halfwidth\fi
\vbox{%
\ifdim\myboxwidth=\hsize
\setbox\onelinebox=\hbox{%
\vbox{\hbox{%
$\Pi_{18,11}$ spans $L_{16.13}$%
}\hbox{%
$36322|223632232|2322\rtimes D_{2}$%
}%
}%
\hfill\copy\matricesbox
}%
\ifdim\wd\onelinebox>\myboxwidth
\hbox to \myboxwidth{%
$\Pi_{18,11}$ spans $L_{16.13}$%
\hfil
$36322|223632232|2322\rtimes D_{2}$%
}%
\box\matricesbox
\else
\hbox to \myboxwidth{%
\unhbox\onelinebox
}%
\fi
\else
\hbox to \myboxwidth{%
$\Pi_{18,11}$ spans $L_{16.13}$%
\hfil}%
\hbox to \myboxwidth{%
$36322|223632232|2322\rtimes D_{2}$%
\hfil}%
\box\matricesbox
\fi
}%
\hfill\discretionary{}{}{}%
\setbox\matricesbox=\hbox{%
{$\left[\!\llap{\phantom{%
\begingroup \smaller\smaller\smaller\begin{tabular}{@{}c@{}}%
\phantom{0}\\\phantom{0}\\\phantom{0}
\end{tabular}\endgroup%
}}\right.$}%
\begingroup \smaller\smaller\smaller\begin{tabular}{@{}c@{}}%
-1\\\phantom{0}\\\phantom{0}
\end{tabular}\endgroup%
\kern3pt%
\begingroup \smaller\smaller\smaller\begin{tabular}{@{}c@{}}%
\phantom{0}\\45/2\\\phantom{0}
\end{tabular}\endgroup%
\kern3pt%
\begingroup \smaller\smaller\smaller\begin{tabular}{@{}c@{}}%
\phantom{0}\\\phantom{0}\\15/2
\end{tabular}\endgroup%
{$\left.\llap{\phantom{%
\begingroup \smaller\smaller\smaller\begin{tabular}{@{}c@{}}%
\phantom{0}\\\phantom{0}\\\phantom{0}
\end{tabular}\endgroup%
}}\!\right]$}%
{$\left[\!\llap{\phantom{%
\begingroup \smaller\smaller\smaller\begin{tabular}{@{}c@{}}%
0\\0\\0
\end{tabular}\endgroup%
}}\right.$}%
\begingroup \smaller\smaller\smaller\begin{tabular}{@{}c@{}}%
90\\19\\3
\end{tabular}\endgroup%
\kern3pt%
\begingroup \smaller\smaller\smaller\begin{tabular}{@{}c@{}}%
5\\1\\1
\end{tabular}\endgroup%
\kern3pt%
\begingroup \smaller\smaller\smaller\begin{tabular}{@{}c@{}}%
9\\1\\3
\end{tabular}\endgroup%
\kern3pt%
\begingroup \smaller\smaller\smaller\begin{tabular}{@{}c@{}}%
30\\1\\11
\end{tabular}\endgroup%
\kern3pt%
\begingroup \smaller\smaller\smaller\begin{tabular}{@{}c@{}}%
30\\-1\\11
\end{tabular}\endgroup%
\kern3pt%
\begingroup \smaller\smaller\smaller\begin{tabular}{@{}c@{}}%
90\\-8\\30
\end{tabular}\endgroup%
\kern3pt%
\begingroup \smaller\smaller\smaller\begin{tabular}{@{}c@{}}%
90\\-11\\27
\end{tabular}\endgroup%
\kern3pt%
\begingroup \smaller\smaller\smaller\begin{tabular}{@{}c@{}}%
5\\-1\\1
\end{tabular}\endgroup%
\kern3pt%
\begingroup \smaller\smaller\smaller\begin{tabular}{@{}c@{}}%
90\\-19\\3
\end{tabular}\endgroup%
{$\left.\llap{\phantom{%
\begingroup \smaller\smaller\smaller\begin{tabular}{@{}c@{}}%
0\\0\\0
\end{tabular}\endgroup%
}}\!\right]$}%
}%
\ifdim\wd\matricesbox>\halfwidth\myboxwidth=\hsize\else\myboxwidth=\halfwidth\fi
\vbox{%
\ifdim\myboxwidth=\hsize
\setbox\onelinebox=\hbox{%
\vbox{\hbox{%
$\Pi_{18,12}$ spans $L_{16.13}$%
}\hbox{%
$363222\slashthree22236322\slashthree22\rtimes D_{2}$%
}%
}%
\hfill\copy\matricesbox
}%
\ifdim\wd\onelinebox>\myboxwidth
\hbox to \myboxwidth{%
$\Pi_{18,12}$ spans $L_{16.13}$%
\hfil
$363222\slashthree22236322\slashthree22\rtimes D_{2}$%
}%
\box\matricesbox
\else
\hbox to \myboxwidth{%
\unhbox\onelinebox
}%
\fi
\else
\hbox to \myboxwidth{%
$\Pi_{18,12}$ spans $L_{16.13}$%
\hfil}%
\hbox to \myboxwidth{%
$363222\slashthree22236322\slashthree22\rtimes D_{2}$%
\hfil}%
\box\matricesbox
\fi
}%
\hfill\discretionary{}{}{}%
\setbox\matricesbox=\hbox{%
{$\left[\!\llap{\phantom{%
\begingroup \smaller\smaller\smaller\begin{tabular}{@{}c@{}}%
\phantom{0}\\\phantom{0}\\\phantom{0}
\end{tabular}\endgroup%
}}\right.$}%
\begingroup \smaller\smaller\smaller\begin{tabular}{@{}c@{}}%
-1\\\phantom{0}\\\phantom{0}
\end{tabular}\endgroup%
\kern3pt%
\begingroup \smaller\smaller\smaller\begin{tabular}{@{}c@{}}%
\phantom{0}\\15/2\\\phantom{0}
\end{tabular}\endgroup%
\kern3pt%
\begingroup \smaller\smaller\smaller\begin{tabular}{@{}c@{}}%
\phantom{0}\\\phantom{0}\\45/2
\end{tabular}\endgroup%
{$\left.\llap{\phantom{%
\begingroup \smaller\smaller\smaller\begin{tabular}{@{}c@{}}%
\phantom{0}\\\phantom{0}\\\phantom{0}
\end{tabular}\endgroup%
}}\!\right]$}%
{$\left[\!\llap{\phantom{%
\begingroup \smaller\smaller\smaller\begin{tabular}{@{}c@{}}%
0\\0\\0
\end{tabular}\endgroup%
}}\right.$}%
\begingroup \smaller\smaller\smaller\begin{tabular}{@{}c@{}}%
5\\2\\0
\end{tabular}\endgroup%
\kern3pt%
\begingroup \smaller\smaller\smaller\begin{tabular}{@{}c@{}}%
90\\30\\8
\end{tabular}\endgroup%
\kern3pt%
\begingroup \smaller\smaller\smaller\begin{tabular}{@{}c@{}}%
90\\27\\11
\end{tabular}\endgroup%
\kern3pt%
\begingroup \smaller\smaller\smaller\begin{tabular}{@{}c@{}}%
5\\1\\1
\end{tabular}\endgroup%
\kern3pt%
\begingroup \smaller\smaller\smaller\begin{tabular}{@{}c@{}}%
9\\0\\2
\end{tabular}\endgroup%
\kern3pt%
\begingroup \smaller\smaller\smaller\begin{tabular}{@{}c@{}}%
30\\-4\\6
\end{tabular}\endgroup%
\kern3pt%
\begingroup \smaller\smaller\smaller\begin{tabular}{@{}c@{}}%
30\\-7\\5
\end{tabular}\endgroup%
\kern3pt%
\begingroup \smaller\smaller\smaller\begin{tabular}{@{}c@{}}%
90\\-27\\11
\end{tabular}\endgroup%
\kern3pt%
\begingroup \smaller\smaller\smaller\begin{tabular}{@{}c@{}}%
90\\-30\\8
\end{tabular}\endgroup%
\kern3pt%
\begingroup \smaller\smaller\smaller\begin{tabular}{@{}c@{}}%
5\\-2\\0
\end{tabular}\endgroup%
{$\left.\llap{\phantom{%
\begingroup \smaller\smaller\smaller\begin{tabular}{@{}c@{}}%
0\\0\\0
\end{tabular}\endgroup%
}}\!\right]$}%
}%
\ifdim\wd\matricesbox>\halfwidth\myboxwidth=\hsize\else\myboxwidth=\halfwidth\fi
\vbox{%
\ifdim\myboxwidth=\hsize
\setbox\onelinebox=\hbox{%
\vbox{\hbox{%
$\Pi_{18,13}$ spans $L_{16.13}$%
}\hbox{%
$36322232|232223632|2\rtimes D_{2}$%
}%
}%
\hfill\copy\matricesbox
}%
\ifdim\wd\onelinebox>\myboxwidth
\hbox to \myboxwidth{%
$\Pi_{18,13}$ spans $L_{16.13}$%
\hfil
$36322232|232223632|2\rtimes D_{2}$%
}%
\box\matricesbox
\else
\hbox to \myboxwidth{%
\unhbox\onelinebox
}%
\fi
\else
\hbox to \myboxwidth{%
$\Pi_{18,13}$ spans $L_{16.13}$%
\hfil}%
\hbox to \myboxwidth{%
$36322232|232223632|2\rtimes D_{2}$%
\hfil}%
\box\matricesbox
\fi
}%
\hfill\discretionary{}{}{}%
\setbox\matricesbox=\hbox{%
{$\left[\!\llap{\phantom{%
\begingroup \smaller\smaller\smaller\begin{tabular}{@{}c@{}}%
\phantom{0}\\\phantom{0}\\\phantom{0}
\end{tabular}\endgroup%
}}\right.$}%
\begingroup \smaller\smaller\smaller\begin{tabular}{@{}c@{}}%
-1\\\phantom{0}\\\phantom{0}
\end{tabular}\endgroup%
\kern3pt%
\begingroup \smaller\smaller\smaller\begin{tabular}{@{}c@{}}%
\phantom{0}\\45/2\\\phantom{0}
\end{tabular}\endgroup%
\kern3pt%
\begingroup \smaller\smaller\smaller\begin{tabular}{@{}c@{}}%
\phantom{0}\\\phantom{0}\\15/2
\end{tabular}\endgroup%
{$\left.\llap{\phantom{%
\begingroup \smaller\smaller\smaller\begin{tabular}{@{}c@{}}%
\phantom{0}\\\phantom{0}\\\phantom{0}
\end{tabular}\endgroup%
}}\!\right]$}%
{$\left[\!\llap{\phantom{%
\begingroup \smaller\smaller\smaller\begin{tabular}{@{}c@{}}%
0\\0\\0
\end{tabular}\endgroup%
}}\right.$}%
\begingroup \smaller\smaller\smaller\begin{tabular}{@{}c@{}}%
90\\19\\3
\end{tabular}\endgroup%
\kern3pt%
\begingroup \smaller\smaller\smaller\begin{tabular}{@{}c@{}}%
30\\6\\4
\end{tabular}\endgroup%
\kern3pt%
\begingroup \smaller\smaller\smaller\begin{tabular}{@{}c@{}}%
30\\5\\7
\end{tabular}\endgroup%
\kern3pt%
\begingroup \smaller\smaller\smaller\begin{tabular}{@{}c@{}}%
9\\1\\3
\end{tabular}\endgroup%
\kern3pt%
\begingroup \smaller\smaller\smaller\begin{tabular}{@{}c@{}}%
5\\0\\2
\end{tabular}\endgroup%
\kern3pt%
\begingroup \smaller\smaller\smaller\begin{tabular}{@{}c@{}}%
90\\-8\\30
\end{tabular}\endgroup%
\kern3pt%
\begingroup \smaller\smaller\smaller\begin{tabular}{@{}c@{}}%
90\\-11\\27
\end{tabular}\endgroup%
\kern3pt%
\begingroup \smaller\smaller\smaller\begin{tabular}{@{}c@{}}%
5\\-1\\1
\end{tabular}\endgroup%
\kern3pt%
\begingroup \smaller\smaller\smaller\begin{tabular}{@{}c@{}}%
90\\-19\\3
\end{tabular}\endgroup%
{$\left.\llap{\phantom{%
\begingroup \smaller\smaller\smaller\begin{tabular}{@{}c@{}}%
0\\0\\0
\end{tabular}\endgroup%
}}\!\right]$}%
}%
\ifdim\wd\matricesbox>\halfwidth\myboxwidth=\hsize\else\myboxwidth=\halfwidth\fi
\vbox{%
\ifdim\myboxwidth=\hsize
\setbox\onelinebox=\hbox{%
\vbox{\hbox{%
$\Pi_{18,14}$ spans $L_{16.13}$%
}\hbox{%
$\slashthree63222322\slashthree22322236\rtimes D_{2}$%
}%
}%
\hfill\copy\matricesbox
}%
\ifdim\wd\onelinebox>\myboxwidth
\hbox to \myboxwidth{%
$\Pi_{18,14}$ spans $L_{16.13}$%
\hfil
$\slashthree63222322\slashthree22322236\rtimes D_{2}$%
}%
\box\matricesbox
\else
\hbox to \myboxwidth{%
\unhbox\onelinebox
}%
\fi
\else
\hbox to \myboxwidth{%
$\Pi_{18,14}$ spans $L_{16.13}$%
\hfil}%
\hbox to \myboxwidth{%
$\slashthree63222322\slashthree22322236\rtimes D_{2}$%
\hfil}%
\box\matricesbox
\fi
}%
\hfill\discretionary{}{}{}%
\setbox\matricesbox=\hbox{%
{$\left[\!\llap{\phantom{%
\begingroup \smaller\smaller\smaller\begin{tabular}{@{}c@{}}%
\phantom{0}\\\phantom{0}\\\phantom{0}
\end{tabular}\endgroup%
}}\right.$}%
\begingroup \smaller\smaller\smaller\begin{tabular}{@{}c@{}}%
-1\\\phantom{0}\\\phantom{0}
\end{tabular}\endgroup%
\kern3pt%
\begingroup \smaller\smaller\smaller\begin{tabular}{@{}c@{}}%
\phantom{0}\\15/2\\\phantom{0}
\end{tabular}\endgroup%
\kern3pt%
\begingroup \smaller\smaller\smaller\begin{tabular}{@{}c@{}}%
\phantom{0}\\\phantom{0}\\45/2
\end{tabular}\endgroup%
{$\left.\llap{\phantom{%
\begingroup \smaller\smaller\smaller\begin{tabular}{@{}c@{}}%
\phantom{0}\\\phantom{0}\\\phantom{0}
\end{tabular}\endgroup%
}}\!\right]$}%
{$\left[\!\llap{\phantom{%
\begingroup \smaller\smaller\smaller\begin{tabular}{@{}c@{}}%
0\\0\\0
\end{tabular}\endgroup%
}}\right.$}%
\begingroup \smaller\smaller\smaller\begin{tabular}{@{}c@{}}%
5\\2\\0
\end{tabular}\endgroup%
\kern3pt%
\begingroup \smaller\smaller\smaller\begin{tabular}{@{}c@{}}%
9\\3\\1
\end{tabular}\endgroup%
\kern3pt%
\begingroup \smaller\smaller\smaller\begin{tabular}{@{}c@{}}%
30\\7\\5
\end{tabular}\endgroup%
\kern3pt%
\begingroup \smaller\smaller\smaller\begin{tabular}{@{}c@{}}%
30\\4\\6
\end{tabular}\endgroup%
\kern3pt%
\begingroup \smaller\smaller\smaller\begin{tabular}{@{}c@{}}%
9\\0\\2
\end{tabular}\endgroup%
\kern3pt%
\begingroup \smaller\smaller\smaller\begin{tabular}{@{}c@{}}%
30\\-4\\6
\end{tabular}\endgroup%
\kern3pt%
\begingroup \smaller\smaller\smaller\begin{tabular}{@{}c@{}}%
30\\-7\\5
\end{tabular}\endgroup%
\kern3pt%
\begingroup \smaller\smaller\smaller\begin{tabular}{@{}c@{}}%
90\\-27\\11
\end{tabular}\endgroup%
\kern3pt%
\begingroup \smaller\smaller\smaller\begin{tabular}{@{}c@{}}%
90\\-30\\8
\end{tabular}\endgroup%
\kern3pt%
\begingroup \smaller\smaller\smaller\begin{tabular}{@{}c@{}}%
5\\-2\\0
\end{tabular}\endgroup%
{$\left.\llap{\phantom{%
\begingroup \smaller\smaller\smaller\begin{tabular}{@{}c@{}}%
0\\0\\0
\end{tabular}\endgroup%
}}\!\right]$}%
}%
\ifdim\wd\matricesbox>\halfwidth\myboxwidth=\hsize\else\myboxwidth=\halfwidth\fi
\vbox{%
\ifdim\myboxwidth=\hsize
\setbox\onelinebox=\hbox{%
\vbox{\hbox{%
$\Pi_{18,15}$ spans $L_{16.13}$%
}\hbox{%
$36322322|223223632|2\rtimes D_{2}$%
}%
}%
\hfill\copy\matricesbox
}%
\ifdim\wd\onelinebox>\myboxwidth
\hbox to \myboxwidth{%
$\Pi_{18,15}$ spans $L_{16.13}$%
\hfil
$36322322|223223632|2\rtimes D_{2}$%
}%
\box\matricesbox
\else
\hbox to \myboxwidth{%
\unhbox\onelinebox
}%
\fi
\else
\hbox to \myboxwidth{%
$\Pi_{18,15}$ spans $L_{16.13}$%
\hfil}%
\hbox to \myboxwidth{%
$36322322|223223632|2\rtimes D_{2}$%
\hfil}%
\box\matricesbox
\fi
}%
\hfill\discretionary{}{}{}%
\setbox\matricesbox=\hbox{%
{$\left[\!\llap{\phantom{%
\begingroup \smaller\smaller\smaller\begin{tabular}{@{}c@{}}%
\phantom{0}\\\phantom{0}\\\phantom{0}
\end{tabular}\endgroup%
}}\right.$}%
\begingroup \smaller\smaller\smaller\begin{tabular}{@{}c@{}}%
-1\\\phantom{0}\\\phantom{0}
\end{tabular}\endgroup%
\kern3pt%
\begingroup \smaller\smaller\smaller\begin{tabular}{@{}c@{}}%
\phantom{0}\\45/2\\\phantom{0}
\end{tabular}\endgroup%
\kern3pt%
\begingroup \smaller\smaller\smaller\begin{tabular}{@{}c@{}}%
\phantom{0}\\\phantom{0}\\15/2
\end{tabular}\endgroup%
{$\left.\llap{\phantom{%
\begingroup \smaller\smaller\smaller\begin{tabular}{@{}c@{}}%
\phantom{0}\\\phantom{0}\\\phantom{0}
\end{tabular}\endgroup%
}}\!\right]$}%
{$\left[\!\llap{\phantom{%
\begingroup \smaller\smaller\smaller\begin{tabular}{@{}c@{}}%
0\\0\\0
\end{tabular}\endgroup%
}}\right.$}%
\begingroup \smaller\smaller\smaller\begin{tabular}{@{}c@{}}%
90\\19\\3
\end{tabular}\endgroup%
\kern3pt%
\begingroup \smaller\smaller\smaller\begin{tabular}{@{}c@{}}%
30\\6\\4
\end{tabular}\endgroup%
\kern3pt%
\begingroup \smaller\smaller\smaller\begin{tabular}{@{}c@{}}%
30\\5\\7
\end{tabular}\endgroup%
\kern3pt%
\begingroup \smaller\smaller\smaller\begin{tabular}{@{}c@{}}%
9\\1\\3
\end{tabular}\endgroup%
\kern3pt%
\begingroup \smaller\smaller\smaller\begin{tabular}{@{}c@{}}%
30\\1\\11
\end{tabular}\endgroup%
\kern3pt%
\begingroup \smaller\smaller\smaller\begin{tabular}{@{}c@{}}%
30\\-1\\11
\end{tabular}\endgroup%
\kern3pt%
\begingroup \smaller\smaller\smaller\begin{tabular}{@{}c@{}}%
9\\-1\\3
\end{tabular}\endgroup%
\kern3pt%
\begingroup \smaller\smaller\smaller\begin{tabular}{@{}c@{}}%
5\\-1\\1
\end{tabular}\endgroup%
\kern3pt%
\begingroup \smaller\smaller\smaller\begin{tabular}{@{}c@{}}%
90\\-19\\3
\end{tabular}\endgroup%
{$\left.\llap{\phantom{%
\begingroup \smaller\smaller\smaller\begin{tabular}{@{}c@{}}%
0\\0\\0
\end{tabular}\endgroup%
}}\!\right]$}%
}%
\ifdim\wd\matricesbox>\halfwidth\myboxwidth=\hsize\else\myboxwidth=\halfwidth\fi
\vbox{%
\ifdim\myboxwidth=\hsize
\setbox\onelinebox=\hbox{%
\vbox{\hbox{%
$\Pi_{18,16}$ spans $L_{16.13}$%
}\hbox{%
$\slashthree63223222\slashthree22232236\rtimes D_{2}$%
}%
}%
\hfill\copy\matricesbox
}%
\ifdim\wd\onelinebox>\myboxwidth
\hbox to \myboxwidth{%
$\Pi_{18,16}$ spans $L_{16.13}$%
\hfil
$\slashthree63223222\slashthree22232236\rtimes D_{2}$%
}%
\box\matricesbox
\else
\hbox to \myboxwidth{%
\unhbox\onelinebox
}%
\fi
\else
\hbox to \myboxwidth{%
$\Pi_{18,16}$ spans $L_{16.13}$%
\hfil}%
\hbox to \myboxwidth{%
$\slashthree63223222\slashthree22232236\rtimes D_{2}$%
\hfil}%
\box\matricesbox
\fi
}%
\hfill\discretionary{}{}{}%
\setbox\matricesbox=\hbox{%
{$\left[\!\llap{\phantom{%
\begingroup \smaller\smaller\smaller\begin{tabular}{@{}c@{}}%
\phantom{0}\\\phantom{0}\\\phantom{0}
\end{tabular}\endgroup%
}}\right.$}%
\begingroup \smaller\smaller\smaller\begin{tabular}{@{}c@{}}%
-1\\\phantom{0}\\\phantom{0}
\end{tabular}\endgroup%
\kern3pt%
\begingroup \smaller\smaller\smaller\begin{tabular}{@{}c@{}}%
\phantom{0}\\45/2\\\phantom{0}
\end{tabular}\endgroup%
\kern3pt%
\begingroup \smaller\smaller\smaller\begin{tabular}{@{}c@{}}%
\phantom{0}\\\phantom{0}\\15/2
\end{tabular}\endgroup%
{$\left.\llap{\phantom{%
\begingroup \smaller\smaller\smaller\begin{tabular}{@{}c@{}}%
\phantom{0}\\\phantom{0}\\\phantom{0}
\end{tabular}\endgroup%
}}\!\right]$}%
{$\left[\!\llap{\phantom{%
\begingroup \smaller\smaller\smaller\begin{tabular}{@{}c@{}}%
0\\0\\0
\end{tabular}\endgroup%
}}\right.$}%
\begingroup \smaller\smaller\smaller\begin{tabular}{@{}c@{}}%
9\\2\\0
\end{tabular}\endgroup%
\kern3pt%
\begingroup \smaller\smaller\smaller\begin{tabular}{@{}c@{}}%
30\\6\\4
\end{tabular}\endgroup%
\kern3pt%
\begingroup \smaller\smaller\smaller\begin{tabular}{@{}c@{}}%
30\\5\\7
\end{tabular}\endgroup%
\kern3pt%
\begingroup \smaller\smaller\smaller\begin{tabular}{@{}c@{}}%
9\\1\\3
\end{tabular}\endgroup%
\kern3pt%
\begingroup \smaller\smaller\smaller\begin{tabular}{@{}c@{}}%
30\\1\\11
\end{tabular}\endgroup%
\kern3pt%
\begingroup \smaller\smaller\smaller\begin{tabular}{@{}c@{}}%
30\\-1\\11
\end{tabular}\endgroup%
\kern3pt%
\begingroup \smaller\smaller\smaller\begin{tabular}{@{}c@{}}%
90\\-8\\30
\end{tabular}\endgroup%
\kern3pt%
\begingroup \smaller\smaller\smaller\begin{tabular}{@{}c@{}}%
90\\-11\\27
\end{tabular}\endgroup%
\kern3pt%
\begingroup \smaller\smaller\smaller\begin{tabular}{@{}c@{}}%
5\\-1\\1
\end{tabular}\endgroup%
\kern3pt%
\begingroup \smaller\smaller\smaller\begin{tabular}{@{}c@{}}%
9\\-2\\0
\end{tabular}\endgroup%
{$\left.\llap{\phantom{%
\begingroup \smaller\smaller\smaller\begin{tabular}{@{}c@{}}%
0\\0\\0
\end{tabular}\endgroup%
}}\!\right]$}%
}%
\ifdim\wd\matricesbox>\halfwidth\myboxwidth=\hsize\else\myboxwidth=\halfwidth\fi
\vbox{%
\ifdim\myboxwidth=\hsize
\setbox\onelinebox=\hbox{%
\vbox{\hbox{%
$\Pi_{18,17}$ spans $L_{16.13}$%
}\hbox{%
$3632232|232236322|22\rtimes D_{2}$%
}%
}%
\hfill\copy\matricesbox
}%
\ifdim\wd\onelinebox>\myboxwidth
\hbox to \myboxwidth{%
$\Pi_{18,17}$ spans $L_{16.13}$%
\hfil
$3632232|232236322|22\rtimes D_{2}$%
}%
\box\matricesbox
\else
\hbox to \myboxwidth{%
\unhbox\onelinebox
}%
\fi
\else
\hbox to \myboxwidth{%
$\Pi_{18,17}$ spans $L_{16.13}$%
\hfil}%
\hbox to \myboxwidth{%
$3632232|232236322|22\rtimes D_{2}$%
\hfil}%
\box\matricesbox
\fi
}%
\hfill\discretionary{}{}{}%
\setbox\matricesbox=\hbox{%
{$\left[\!\llap{\phantom{%
\begingroup \smaller\smaller\smaller\begin{tabular}{@{}c@{}}%
\phantom{0}\\\phantom{0}\\\phantom{0}
\end{tabular}\endgroup%
}}\right.$}%
\begingroup \smaller\smaller\smaller\begin{tabular}{@{}c@{}}%
-1\\\phantom{0}\\\phantom{0}
\end{tabular}\endgroup%
\kern3pt%
\begingroup \smaller\smaller\smaller\begin{tabular}{@{}c@{}}%
\phantom{0}\\15/2\\\phantom{0}
\end{tabular}\endgroup%
\kern3pt%
\begingroup \smaller\smaller\smaller\begin{tabular}{@{}c@{}}%
\phantom{0}\\\phantom{0}\\45/2
\end{tabular}\endgroup%
{$\left.\llap{\phantom{%
\begingroup \smaller\smaller\smaller\begin{tabular}{@{}c@{}}%
\phantom{0}\\\phantom{0}\\\phantom{0}
\end{tabular}\endgroup%
}}\!\right]$}%
{$\left[\!\llap{\phantom{%
\begingroup \smaller\smaller\smaller\begin{tabular}{@{}c@{}}%
0\\0\\0
\end{tabular}\endgroup%
}}\right.$}%
\begingroup \smaller\smaller\smaller\begin{tabular}{@{}c@{}}%
30\\-11\\-1
\end{tabular}\endgroup%
\kern3pt%
\begingroup \smaller\smaller\smaller\begin{tabular}{@{}c@{}}%
9\\-3\\-1
\end{tabular}\endgroup%
\kern3pt%
\begingroup \smaller\smaller\smaller\begin{tabular}{@{}c@{}}%
30\\-7\\-5
\end{tabular}\endgroup%
\kern3pt%
\begingroup \smaller\smaller\smaller\begin{tabular}{@{}c@{}}%
30\\-4\\-6
\end{tabular}\endgroup%
\kern3pt%
\begingroup \smaller\smaller\smaller\begin{tabular}{@{}c@{}}%
90\\-3\\-19
\end{tabular}\endgroup%
\kern3pt%
\begingroup \smaller\smaller\smaller\begin{tabular}{@{}c@{}}%
90\\3\\-19
\end{tabular}\endgroup%
\kern3pt%
\begingroup \smaller\smaller\smaller\begin{tabular}{@{}c@{}}%
5\\1\\-1
\end{tabular}\endgroup%
\kern3pt%
\begingroup \smaller\smaller\smaller\begin{tabular}{@{}c@{}}%
9\\3\\-1
\end{tabular}\endgroup%
\kern3pt%
\begingroup \smaller\smaller\smaller\begin{tabular}{@{}c@{}}%
30\\11\\-1
\end{tabular}\endgroup%
{$\left.\llap{\phantom{%
\begingroup \smaller\smaller\smaller\begin{tabular}{@{}c@{}}%
0\\0\\0
\end{tabular}\endgroup%
}}\!\right]$}%
}%
\ifdim\wd\matricesbox>\halfwidth\myboxwidth=\hsize\else\myboxwidth=\halfwidth\fi
\vbox{%
\ifdim\myboxwidth=\hsize
\setbox\onelinebox=\hbox{%
\vbox{\hbox{%
$\Pi_{18,18}$ spans $L_{16.13}$%
}\hbox{%
$36322\slashthree22363222\slashthree222\rtimes D_{2}$%
}%
}%
\hfill\copy\matricesbox
}%
\ifdim\wd\onelinebox>\myboxwidth
\hbox to \myboxwidth{%
$\Pi_{18,18}$ spans $L_{16.13}$%
\hfil
$36322\slashthree22363222\slashthree222\rtimes D_{2}$%
}%
\box\matricesbox
\else
\hbox to \myboxwidth{%
\unhbox\onelinebox
}%
\fi
\else
\hbox to \myboxwidth{%
$\Pi_{18,18}$ spans $L_{16.13}$%
\hfil}%
\hbox to \myboxwidth{%
$36322\slashthree22363222\slashthree222\rtimes D_{2}$%
\hfil}%
\box\matricesbox
\fi
}%
\hfill\discretionary{}{}{}%
\setbox\matricesbox=\hbox{%
{$\left[\!\llap{\phantom{%
\begingroup \smaller\smaller\smaller\begin{tabular}{@{}c@{}}%
\phantom{0}\\\phantom{0}\\\phantom{0}
\end{tabular}\endgroup%
}}\right.$}%
\begingroup \smaller\smaller\smaller\begin{tabular}{@{}c@{}}%
-1\\\phantom{0}\\\phantom{0}
\end{tabular}\endgroup%
\kern3pt%
\begingroup \smaller\smaller\smaller\begin{tabular}{@{}c@{}}%
\phantom{0}\\45/2\\\phantom{0}
\end{tabular}\endgroup%
\kern3pt%
\begingroup \smaller\smaller\smaller\begin{tabular}{@{}c@{}}%
\phantom{0}\\\phantom{0}\\15/2
\end{tabular}\endgroup%
{$\left.\llap{\phantom{%
\begingroup \smaller\smaller\smaller\begin{tabular}{@{}c@{}}%
\phantom{0}\\\phantom{0}\\\phantom{0}
\end{tabular}\endgroup%
}}\!\right]$}%
{$\left[\!\llap{\phantom{%
\begingroup \smaller\smaller\smaller\begin{tabular}{@{}c@{}}%
0\\0\\0
\end{tabular}\endgroup%
}}\right.$}%
\begingroup \smaller\smaller\smaller\begin{tabular}{@{}c@{}}%
9\\-2\\0
\end{tabular}\endgroup%
\kern3pt%
\begingroup \smaller\smaller\smaller\begin{tabular}{@{}c@{}}%
30\\-6\\4
\end{tabular}\endgroup%
\kern3pt%
\begingroup \smaller\smaller\smaller\begin{tabular}{@{}c@{}}%
30\\-5\\7
\end{tabular}\endgroup%
\kern3pt%
\begingroup \smaller\smaller\smaller\begin{tabular}{@{}c@{}}%
90\\-11\\27
\end{tabular}\endgroup%
\kern3pt%
\begingroup \smaller\smaller\smaller\begin{tabular}{@{}c@{}}%
90\\-8\\30
\end{tabular}\endgroup%
\kern3pt%
\begingroup \smaller\smaller\smaller\begin{tabular}{@{}c@{}}%
5\\0\\2
\end{tabular}\endgroup%
\kern3pt%
\begingroup \smaller\smaller\smaller\begin{tabular}{@{}c@{}}%
9\\1\\3
\end{tabular}\endgroup%
\kern3pt%
\begingroup \smaller\smaller\smaller\begin{tabular}{@{}c@{}}%
30\\5\\7
\end{tabular}\endgroup%
\kern3pt%
\begingroup \smaller\smaller\smaller\begin{tabular}{@{}c@{}}%
30\\6\\4
\end{tabular}\endgroup%
\kern3pt%
\begingroup \smaller\smaller\smaller\begin{tabular}{@{}c@{}}%
9\\2\\0
\end{tabular}\endgroup%
{$\left.\llap{\phantom{%
\begingroup \smaller\smaller\smaller\begin{tabular}{@{}c@{}}%
0\\0\\0
\end{tabular}\endgroup%
}}\!\right]$}%
}%
\ifdim\wd\matricesbox>\halfwidth\myboxwidth=\hsize\else\myboxwidth=\halfwidth\fi
\vbox{%
\ifdim\myboxwidth=\hsize
\setbox\onelinebox=\hbox{%
\vbox{\hbox{%
$\Pi_{18,19}$ spans $L_{16.13}$%
}\hbox{%
$3632|236322232|23222\rtimes D_{2}$%
}%
}%
\hfill\copy\matricesbox
}%
\ifdim\wd\onelinebox>\myboxwidth
\hbox to \myboxwidth{%
$\Pi_{18,19}$ spans $L_{16.13}$%
\hfil
$3632|236322232|23222\rtimes D_{2}$%
}%
\box\matricesbox
\else
\hbox to \myboxwidth{%
\unhbox\onelinebox
}%
\fi
\else
\hbox to \myboxwidth{%
$\Pi_{18,19}$ spans $L_{16.13}$%
\hfil}%
\hbox to \myboxwidth{%
$3632|236322232|23222\rtimes D_{2}$%
\hfil}%
\box\matricesbox
\fi
}%
\hfill\discretionary{}{}{}%
\setbox\matricesbox=\hbox{%
{$\left[\!\llap{\phantom{%
\begingroup \smaller\smaller\smaller\begin{tabular}{@{}c@{}}%
\phantom{0}\\\phantom{0}\\\phantom{0}
\end{tabular}\endgroup%
}}\right.$}%
\begingroup \smaller\smaller\smaller\begin{tabular}{@{}c@{}}%
-1\\\phantom{0}\\\phantom{0}
\end{tabular}\endgroup%
\kern3pt%
\begingroup \smaller\smaller\smaller\begin{tabular}{@{}c@{}}%
\phantom{0}\\45/2\\\phantom{0}
\end{tabular}\endgroup%
\kern3pt%
\begingroup \smaller\smaller\smaller\begin{tabular}{@{}c@{}}%
\phantom{0}\\\phantom{0}\\15/2
\end{tabular}\endgroup%
{$\left.\llap{\phantom{%
\begingroup \smaller\smaller\smaller\begin{tabular}{@{}c@{}}%
\phantom{0}\\\phantom{0}\\\phantom{0}
\end{tabular}\endgroup%
}}\!\right]$}%
{$\left[\!\llap{\phantom{%
\begingroup \smaller\smaller\smaller\begin{tabular}{@{}c@{}}%
0\\0\\0
\end{tabular}\endgroup%
}}\right.$}%
\begingroup \smaller\smaller\smaller\begin{tabular}{@{}c@{}}%
9\\-2\\0
\end{tabular}\endgroup%
\kern3pt%
\begingroup \smaller\smaller\smaller\begin{tabular}{@{}c@{}}%
30\\-6\\4
\end{tabular}\endgroup%
\kern3pt%
\begingroup \smaller\smaller\smaller\begin{tabular}{@{}c@{}}%
30\\-5\\7
\end{tabular}\endgroup%
\kern3pt%
\begingroup \smaller\smaller\smaller\begin{tabular}{@{}c@{}}%
90\\-11\\27
\end{tabular}\endgroup%
\kern3pt%
\begingroup \smaller\smaller\smaller\begin{tabular}{@{}c@{}}%
90\\-8\\30
\end{tabular}\endgroup%
\kern3pt%
\begingroup \smaller\smaller\smaller\begin{tabular}{@{}c@{}}%
5\\0\\2
\end{tabular}\endgroup%
\kern3pt%
\begingroup \smaller\smaller\smaller\begin{tabular}{@{}c@{}}%
90\\8\\30
\end{tabular}\endgroup%
\kern3pt%
\begingroup \smaller\smaller\smaller\begin{tabular}{@{}c@{}}%
90\\11\\27
\end{tabular}\endgroup%
\kern3pt%
\begingroup \smaller\smaller\smaller\begin{tabular}{@{}c@{}}%
5\\1\\1
\end{tabular}\endgroup%
\kern3pt%
\begingroup \smaller\smaller\smaller\begin{tabular}{@{}c@{}}%
9\\2\\0
\end{tabular}\endgroup%
{$\left.\llap{\phantom{%
\begingroup \smaller\smaller\smaller\begin{tabular}{@{}c@{}}%
0\\0\\0
\end{tabular}\endgroup%
}}\!\right]$}%
}%
\ifdim\wd\matricesbox>\halfwidth\myboxwidth=\hsize\else\myboxwidth=\halfwidth\fi
\vbox{%
\ifdim\myboxwidth=\hsize
\setbox\onelinebox=\hbox{%
\vbox{\hbox{%
$\Pi_{18,20}$ spans $L_{16.13}$%
}\hbox{%
$3632|236322322|22322\rtimes D_{2}$%
}%
}%
\hfill\copy\matricesbox
}%
\ifdim\wd\onelinebox>\myboxwidth
\hbox to \myboxwidth{%
$\Pi_{18,20}$ spans $L_{16.13}$%
\hfil
$3632|236322322|22322\rtimes D_{2}$%
}%
\box\matricesbox
\else
\hbox to \myboxwidth{%
\unhbox\onelinebox
}%
\fi
\else
\hbox to \myboxwidth{%
$\Pi_{18,20}$ spans $L_{16.13}$%
\hfil}%
\hbox to \myboxwidth{%
$3632|236322322|22322\rtimes D_{2}$%
\hfil}%
\box\matricesbox
\fi
}%
\hfill\discretionary{}{}{}%
\setbox\matricesbox=\hbox{%
{$\left[\!\llap{\phantom{%
\begingroup \smaller\smaller\smaller\begin{tabular}{@{}c@{}}%
\phantom{0}\\\phantom{0}\\\phantom{0}
\end{tabular}\endgroup%
}}\right.$}%
\begingroup \smaller\smaller\smaller\begin{tabular}{@{}c@{}}%
-1\\\phantom{0}\\\phantom{0}
\end{tabular}\endgroup%
\kern3pt%
\begingroup \smaller\smaller\smaller\begin{tabular}{@{}c@{}}%
\phantom{0}\\45/2\\\phantom{0}
\end{tabular}\endgroup%
\kern3pt%
\begingroup \smaller\smaller\smaller\begin{tabular}{@{}c@{}}%
\phantom{0}\\\phantom{0}\\15/2
\end{tabular}\endgroup%
{$\left.\llap{\phantom{%
\begingroup \smaller\smaller\smaller\begin{tabular}{@{}c@{}}%
\phantom{0}\\\phantom{0}\\\phantom{0}
\end{tabular}\endgroup%
}}\!\right]$}%
{$\left[\!\llap{\phantom{%
\begingroup \smaller\smaller\smaller\begin{tabular}{@{}c@{}}%
0\\0\\0
\end{tabular}\endgroup%
}}\right.$}%
\begingroup \smaller\smaller\smaller\begin{tabular}{@{}c@{}}%
90\\19\\3
\end{tabular}\endgroup%
\kern3pt%
\begingroup \smaller\smaller\smaller\begin{tabular}{@{}c@{}}%
30\\6\\4
\end{tabular}\endgroup%
\kern3pt%
\begingroup \smaller\smaller\smaller\begin{tabular}{@{}c@{}}%
30\\5\\7
\end{tabular}\endgroup%
\kern3pt%
\begingroup \smaller\smaller\smaller\begin{tabular}{@{}c@{}}%
90\\11\\27
\end{tabular}\endgroup%
\kern3pt%
\begingroup \smaller\smaller\smaller\begin{tabular}{@{}c@{}}%
90\\8\\30
\end{tabular}\endgroup%
\kern3pt%
\begingroup \smaller\smaller\smaller\begin{tabular}{@{}c@{}}%
5\\0\\2
\end{tabular}\endgroup%
\kern3pt%
\begingroup \smaller\smaller\smaller\begin{tabular}{@{}c@{}}%
9\\-1\\3
\end{tabular}\endgroup%
\kern3pt%
\begingroup \smaller\smaller\smaller\begin{tabular}{@{}c@{}}%
5\\-1\\1
\end{tabular}\endgroup%
\kern3pt%
\begingroup \smaller\smaller\smaller\begin{tabular}{@{}c@{}}%
90\\-19\\3
\end{tabular}\endgroup%
{$\left.\llap{\phantom{%
\begingroup \smaller\smaller\smaller\begin{tabular}{@{}c@{}}%
0\\0\\0
\end{tabular}\endgroup%
}}\!\right]$}%
}%
\ifdim\wd\matricesbox>\halfwidth\myboxwidth=\hsize\else\myboxwidth=\halfwidth\fi
\vbox{%
\ifdim\myboxwidth=\hsize
\setbox\onelinebox=\hbox{%
\vbox{\hbox{%
$\Pi_{18,21}$ spans $L_{16.13}$%
}\hbox{%
$\slashthree63632222\slashthree22223636\rtimes D_{2}$%
}%
}%
\hfill\copy\matricesbox
}%
\ifdim\wd\onelinebox>\myboxwidth
\hbox to \myboxwidth{%
$\Pi_{18,21}$ spans $L_{16.13}$%
\hfil
$\slashthree63632222\slashthree22223636\rtimes D_{2}$%
}%
\box\matricesbox
\else
\hbox to \myboxwidth{%
\unhbox\onelinebox
}%
\fi
\else
\hbox to \myboxwidth{%
$\Pi_{18,21}$ spans $L_{16.13}$%
\hfil}%
\hbox to \myboxwidth{%
$\slashthree63632222\slashthree22223636\rtimes D_{2}$%
\hfil}%
\box\matricesbox
\fi
}%
\hfill\discretionary{}{}{}%
\setbox\matricesbox=\hbox{%
{$\left[\!\llap{\phantom{%
\begingroup \smaller\smaller\smaller\begin{tabular}{@{}c@{}}%
\phantom{0}\\\phantom{0}\\\phantom{0}
\end{tabular}\endgroup%
}}\right.$}%
\begingroup \smaller\smaller\smaller\begin{tabular}{@{}c@{}}%
-1\\\phantom{0}\\\phantom{0}
\end{tabular}\endgroup%
\kern3pt%
\begingroup \smaller\smaller\smaller\begin{tabular}{@{}c@{}}%
\phantom{0}\\15/2\\\phantom{0}
\end{tabular}\endgroup%
\kern3pt%
\begingroup \smaller\smaller\smaller\begin{tabular}{@{}c@{}}%
\phantom{0}\\\phantom{0}\\45/2
\end{tabular}\endgroup%
{$\left.\llap{\phantom{%
\begingroup \smaller\smaller\smaller\begin{tabular}{@{}c@{}}%
\phantom{0}\\\phantom{0}\\\phantom{0}
\end{tabular}\endgroup%
}}\!\right]$}%
{$\left[\!\llap{\phantom{%
\begingroup \smaller\smaller\smaller\begin{tabular}{@{}c@{}}%
0\\0\\0
\end{tabular}\endgroup%
}}\right.$}%
\begingroup \smaller\smaller\smaller\begin{tabular}{@{}c@{}}%
30\\-11\\-1
\end{tabular}\endgroup%
\kern3pt%
\begingroup \smaller\smaller\smaller\begin{tabular}{@{}c@{}}%
90\\-30\\-8
\end{tabular}\endgroup%
\kern3pt%
\begingroup \smaller\smaller\smaller\begin{tabular}{@{}c@{}}%
90\\-27\\-11
\end{tabular}\endgroup%
\kern3pt%
\begingroup \smaller\smaller\smaller\begin{tabular}{@{}c@{}}%
5\\-1\\-1
\end{tabular}\endgroup%
\kern3pt%
\begingroup \smaller\smaller\smaller\begin{tabular}{@{}c@{}}%
9\\0\\-2
\end{tabular}\endgroup%
\kern3pt%
\begingroup \smaller\smaller\smaller\begin{tabular}{@{}c@{}}%
30\\4\\-6
\end{tabular}\endgroup%
\kern3pt%
\begingroup \smaller\smaller\smaller\begin{tabular}{@{}c@{}}%
30\\7\\-5
\end{tabular}\endgroup%
\kern3pt%
\begingroup \smaller\smaller\smaller\begin{tabular}{@{}c@{}}%
9\\3\\-1
\end{tabular}\endgroup%
\kern3pt%
\begingroup \smaller\smaller\smaller\begin{tabular}{@{}c@{}}%
30\\11\\-1
\end{tabular}\endgroup%
{$\left.\llap{\phantom{%
\begingroup \smaller\smaller\smaller\begin{tabular}{@{}c@{}}%
0\\0\\0
\end{tabular}\endgroup%
}}\!\right]$}%
}%
\ifdim\wd\matricesbox>\halfwidth\myboxwidth=\hsize\else\myboxwidth=\halfwidth\fi
\vbox{%
\ifdim\myboxwidth=\hsize
\setbox\onelinebox=\hbox{%
\vbox{\hbox{%
$\Pi_{18,22}$ spans $L_{16.13}$%
}\hbox{%
$36\slashthree63222322\slashthree223222\rtimes D_{2}$%
}%
}%
\hfill\copy\matricesbox
}%
\ifdim\wd\onelinebox>\myboxwidth
\hbox to \myboxwidth{%
$\Pi_{18,22}$ spans $L_{16.13}$%
\hfil
$36\slashthree63222322\slashthree223222\rtimes D_{2}$%
}%
\box\matricesbox
\else
\hbox to \myboxwidth{%
\unhbox\onelinebox
}%
\fi
\else
\hbox to \myboxwidth{%
$\Pi_{18,22}$ spans $L_{16.13}$%
\hfil}%
\hbox to \myboxwidth{%
$36\slashthree63222322\slashthree223222\rtimes D_{2}$%
\hfil}%
\box\matricesbox
\fi
}%
\hfill\discretionary{}{}{}%
\setbox\matricesbox=\hbox{%
{$\left[\!\llap{\phantom{%
\begingroup \smaller\smaller\smaller\begin{tabular}{@{}c@{}}%
\phantom{0}\\\phantom{0}\\\phantom{0}
\end{tabular}\endgroup%
}}\right.$}%
\begingroup \smaller\smaller\smaller\begin{tabular}{@{}c@{}}%
-1\\\phantom{0}\\\phantom{0}
\end{tabular}\endgroup%
\kern3pt%
\begingroup \smaller\smaller\smaller\begin{tabular}{@{}c@{}}%
\phantom{0}\\15/2\\\phantom{0}
\end{tabular}\endgroup%
\kern3pt%
\begingroup \smaller\smaller\smaller\begin{tabular}{@{}c@{}}%
\phantom{0}\\\phantom{0}\\45/2
\end{tabular}\endgroup%
{$\left.\llap{\phantom{%
\begingroup \smaller\smaller\smaller\begin{tabular}{@{}c@{}}%
\phantom{0}\\\phantom{0}\\\phantom{0}
\end{tabular}\endgroup%
}}\!\right]$}%
{$\left[\!\llap{\phantom{%
\begingroup \smaller\smaller\smaller\begin{tabular}{@{}c@{}}%
0\\0\\0
\end{tabular}\endgroup%
}}\right.$}%
\begingroup \smaller\smaller\smaller\begin{tabular}{@{}c@{}}%
30\\-11\\-1
\end{tabular}\endgroup%
\kern3pt%
\begingroup \smaller\smaller\smaller\begin{tabular}{@{}c@{}}%
90\\-30\\-8
\end{tabular}\endgroup%
\kern3pt%
\begingroup \smaller\smaller\smaller\begin{tabular}{@{}c@{}}%
90\\-27\\-11
\end{tabular}\endgroup%
\kern3pt%
\begingroup \smaller\smaller\smaller\begin{tabular}{@{}c@{}}%
5\\-1\\-1
\end{tabular}\endgroup%
\kern3pt%
\begingroup \smaller\smaller\smaller\begin{tabular}{@{}c@{}}%
90\\-3\\-19
\end{tabular}\endgroup%
\kern3pt%
\begingroup \smaller\smaller\smaller\begin{tabular}{@{}c@{}}%
90\\3\\-19
\end{tabular}\endgroup%
\kern3pt%
\begingroup \smaller\smaller\smaller\begin{tabular}{@{}c@{}}%
5\\1\\-1
\end{tabular}\endgroup%
\kern3pt%
\begingroup \smaller\smaller\smaller\begin{tabular}{@{}c@{}}%
9\\3\\-1
\end{tabular}\endgroup%
\kern3pt%
\begingroup \smaller\smaller\smaller\begin{tabular}{@{}c@{}}%
30\\11\\-1
\end{tabular}\endgroup%
{$\left.\llap{\phantom{%
\begingroup \smaller\smaller\smaller\begin{tabular}{@{}c@{}}%
0\\0\\0
\end{tabular}\endgroup%
}}\!\right]$}%
}%
\ifdim\wd\matricesbox>\halfwidth\myboxwidth=\hsize\else\myboxwidth=\halfwidth\fi
\vbox{%
\ifdim\myboxwidth=\hsize
\setbox\onelinebox=\hbox{%
\vbox{\hbox{%
$\Pi_{18,23}$ spans $L_{16.13}$%
}\hbox{%
$36\slashthree63223222\slashthree222322\rtimes D_{2}$%
}%
}%
\hfill\copy\matricesbox
}%
\ifdim\wd\onelinebox>\myboxwidth
\hbox to \myboxwidth{%
$\Pi_{18,23}$ spans $L_{16.13}$%
\hfil
$36\slashthree63223222\slashthree222322\rtimes D_{2}$%
}%
\box\matricesbox
\else
\hbox to \myboxwidth{%
\unhbox\onelinebox
}%
\fi
\else
\hbox to \myboxwidth{%
$\Pi_{18,23}$ spans $L_{16.13}$%
\hfil}%
\hbox to \myboxwidth{%
$36\slashthree63223222\slashthree222322\rtimes D_{2}$%
\hfil}%
\box\matricesbox
\fi
}%
\hfill\discretionary{}{}{}%
\setbox\matricesbox=\hbox{%
{$\left[\!\llap{\phantom{%
\begingroup \smaller\smaller\smaller\begin{tabular}{@{}c@{}}%
\phantom{0}\\\phantom{0}\\\phantom{0}
\end{tabular}\endgroup%
}}\right.$}%
\begingroup \smaller\smaller\smaller\begin{tabular}{@{}c@{}}%
-1\\\phantom{0}\\\phantom{0}
\end{tabular}\endgroup%
\kern3pt%
\begingroup \smaller\smaller\smaller\begin{tabular}{@{}c@{}}%
\phantom{0}\\18\\\phantom{0}
\end{tabular}\endgroup%
\kern3pt%
\begingroup \smaller\smaller\smaller\begin{tabular}{@{}c@{}}%
\phantom{0}\\\phantom{0}\\18
\end{tabular}\endgroup%
{$\left.\llap{\phantom{%
\begingroup \smaller\smaller\smaller\begin{tabular}{@{}c@{}}%
\phantom{0}\\\phantom{0}\\\phantom{0}
\end{tabular}\endgroup%
}}\!\right]$}%
{$\left[\!\llap{\phantom{%
\begingroup \smaller\smaller\smaller\begin{tabular}{@{}c@{}}%
0\\0\\0
\end{tabular}\endgroup%
}}\right.$}%
\begingroup \smaller\smaller\smaller\begin{tabular}{@{}c@{}}%
8\\2\\0
\end{tabular}\endgroup%
\kern3pt%
\begingroup \smaller\smaller\smaller\begin{tabular}{@{}c@{}}%
72\\16\\-6
\end{tabular}\endgroup%
\kern3pt%
\begingroup \smaller\smaller\smaller\begin{tabular}{@{}c@{}}%
36\\7\\-5
\end{tabular}\endgroup%
\kern3pt%
\begingroup \smaller\smaller\smaller\begin{tabular}{@{}c@{}}%
9\\1\\-2
\end{tabular}\endgroup%
\kern3pt%
\begingroup \smaller\smaller\smaller\begin{tabular}{@{}c@{}}%
8\\0\\-2
\end{tabular}\endgroup%
\kern3pt%
\begingroup \smaller\smaller\smaller\begin{tabular}{@{}c@{}}%
72\\-6\\-16
\end{tabular}\endgroup%
\kern3pt%
\begingroup \smaller\smaller\smaller\begin{tabular}{@{}c@{}}%
36\\-5\\-7
\end{tabular}\endgroup%
\kern3pt%
\begingroup \smaller\smaller\smaller\begin{tabular}{@{}c@{}}%
36\\-7\\-5
\end{tabular}\endgroup%
\kern3pt%
\begingroup \smaller\smaller\smaller\begin{tabular}{@{}c@{}}%
72\\-16\\-6
\end{tabular}\endgroup%
\kern3pt%
\begingroup \smaller\smaller\smaller\begin{tabular}{@{}c@{}}%
8\\-2\\0
\end{tabular}\endgroup%
{$\left.\llap{\phantom{%
\begingroup \smaller\smaller\smaller\begin{tabular}{@{}c@{}}%
0\\0\\0
\end{tabular}\endgroup%
}}\!\right]$}%
}%
\ifdim\wd\matricesbox>\halfwidth\myboxwidth=\hsize\else\myboxwidth=\halfwidth\fi
\vbox{%
\ifdim\myboxwidth=\hsize
\setbox\onelinebox=\hbox{%
\vbox{\hbox{%
$\Pi_{18,24}$ spans $L_{142.20}$%
}\hbox{%
$\infty42|24\infty224\infty42|24\infty422\rtimes D_{2}$%
}%
}%
\hfill\copy\matricesbox
}%
\ifdim\wd\onelinebox>\myboxwidth
\hbox to \myboxwidth{%
$\Pi_{18,24}$ spans $L_{142.20}$%
\hfil
$\infty42|24\infty224\infty42|24\infty422\rtimes D_{2}$%
}%
\box\matricesbox
\else
\hbox to \myboxwidth{%
\unhbox\onelinebox
}%
\fi
\else
\hbox to \myboxwidth{%
$\Pi_{18,24}$ spans $L_{142.20}$%
\hfil}%
\hbox to \myboxwidth{%
$\infty42|24\infty224\infty42|24\infty422\rtimes D_{2}$%
\hfil}%
\box\matricesbox
\fi
}%
\hfill\discretionary{}{}{}%
\setbox\matricesbox=\hbox{%
{$\left[\!\llap{\phantom{%
\begingroup \smaller\smaller\smaller\begin{tabular}{@{}c@{}}%
\phantom{0}\\\phantom{0}\\\phantom{0}
\end{tabular}\endgroup%
}}\right.$}%
\begingroup \smaller\smaller\smaller\begin{tabular}{@{}c@{}}%
-1\\\phantom{0}\\\phantom{0}
\end{tabular}\endgroup%
\kern3pt%
\begingroup \smaller\smaller\smaller\begin{tabular}{@{}c@{}}%
\phantom{0}\\9\\\phantom{0}
\end{tabular}\endgroup%
\kern3pt%
\begingroup \smaller\smaller\smaller\begin{tabular}{@{}c@{}}%
\phantom{0}\\\phantom{0}\\9
\end{tabular}\endgroup%
{$\left.\llap{\phantom{%
\begingroup \smaller\smaller\smaller\begin{tabular}{@{}c@{}}%
\phantom{0}\\\phantom{0}\\\phantom{0}
\end{tabular}\endgroup%
}}\!\right]$}%
{$\left[\!\llap{\phantom{%
\begingroup \smaller\smaller\smaller\begin{tabular}{@{}c@{}}%
0\\0\\0
\end{tabular}\endgroup%
}}\right.$}%
\begingroup \smaller\smaller\smaller\begin{tabular}{@{}c@{}}%
36\\12\\-2
\end{tabular}\endgroup%
\kern3pt%
\begingroup \smaller\smaller\smaller\begin{tabular}{@{}c@{}}%
72\\22\\-10
\end{tabular}\endgroup%
\kern3pt%
\begingroup \smaller\smaller\smaller\begin{tabular}{@{}c@{}}%
8\\2\\-2
\end{tabular}\endgroup%
\kern3pt%
\begingroup \smaller\smaller\smaller\begin{tabular}{@{}c@{}}%
72\\10\\-22
\end{tabular}\endgroup%
\kern3pt%
\begingroup \smaller\smaller\smaller\begin{tabular}{@{}c@{}}%
36\\2\\-12
\end{tabular}\endgroup%
\kern3pt%
\begingroup \smaller\smaller\smaller\begin{tabular}{@{}c@{}}%
9\\-1\\-3
\end{tabular}\endgroup%
\kern3pt%
\begingroup \smaller\smaller\smaller\begin{tabular}{@{}c@{}}%
8\\-2\\-2
\end{tabular}\endgroup%
\kern3pt%
\begingroup \smaller\smaller\smaller\begin{tabular}{@{}c@{}}%
72\\-22\\-10
\end{tabular}\endgroup%
\kern3pt%
\begingroup \smaller\smaller\smaller\begin{tabular}{@{}c@{}}%
36\\-12\\-2
\end{tabular}\endgroup%
{$\left.\llap{\phantom{%
\begingroup \smaller\smaller\smaller\begin{tabular}{@{}c@{}}%
0\\0\\0
\end{tabular}\endgroup%
}}\!\right]$}%
}%
\ifdim\wd\matricesbox>\halfwidth\myboxwidth=\hsize\else\myboxwidth=\halfwidth\fi
\vbox{%
\ifdim\myboxwidth=\hsize
\setbox\onelinebox=\hbox{%
\vbox{\hbox{%
$\Pi_{18,25}$ spans $L_{142.20}$%
}\hbox{%
$\infty4224\slashinfty4224\infty224\slashinfty422\rtimes D_{2}$%
}%
}%
\hfill\copy\matricesbox
}%
\ifdim\wd\onelinebox>\myboxwidth
\hbox to \myboxwidth{%
$\Pi_{18,25}$ spans $L_{142.20}$%
\hfil
$\infty4224\slashinfty4224\infty224\slashinfty422\rtimes D_{2}$%
}%
\box\matricesbox
\else
\hbox to \myboxwidth{%
\unhbox\onelinebox
}%
\fi
\else
\hbox to \myboxwidth{%
$\Pi_{18,25}$ spans $L_{142.20}$%
\hfil}%
\hbox to \myboxwidth{%
$\infty4224\slashinfty4224\infty224\slashinfty422\rtimes D_{2}$%
\hfil}%
\box\matricesbox
\fi
}%
\hfill\discretionary{}{}{}%
\setbox\matricesbox=\hbox{%
{$\left[\!\llap{\phantom{%
\begingroup \smaller\smaller\smaller\begin{tabular}{@{}c@{}}%
\phantom{0}\\\phantom{0}\\\phantom{0}
\end{tabular}\endgroup%
}}\right.$}%
\begingroup \smaller\smaller\smaller\begin{tabular}{@{}c@{}}%
-1\\\phantom{0}\\\phantom{0}
\end{tabular}\endgroup%
\kern3pt%
\begingroup \smaller\smaller\smaller\begin{tabular}{@{}c@{}}%
\phantom{0}\\18\\\phantom{0}
\end{tabular}\endgroup%
\kern3pt%
\begingroup \smaller\smaller\smaller\begin{tabular}{@{}c@{}}%
\phantom{0}\\\phantom{0}\\18
\end{tabular}\endgroup%
{$\left.\llap{\phantom{%
\begingroup \smaller\smaller\smaller\begin{tabular}{@{}c@{}}%
\phantom{0}\\\phantom{0}\\\phantom{0}
\end{tabular}\endgroup%
}}\!\right]$}%
{$\left[\!\llap{\phantom{%
\begingroup \smaller\smaller\smaller\begin{tabular}{@{}c@{}}%
0\\0\\0
\end{tabular}\endgroup%
}}\right.$}%
\begingroup \smaller\smaller\smaller\begin{tabular}{@{}c@{}}%
8\\2\\0
\end{tabular}\endgroup%
\kern3pt%
\begingroup \smaller\smaller\smaller\begin{tabular}{@{}c@{}}%
72\\16\\6
\end{tabular}\endgroup%
\kern3pt%
\begingroup \smaller\smaller\smaller\begin{tabular}{@{}c@{}}%
36\\7\\5
\end{tabular}\endgroup%
\kern3pt%
\begingroup \smaller\smaller\smaller\begin{tabular}{@{}c@{}}%
36\\5\\7
\end{tabular}\endgroup%
\kern3pt%
\begingroup \smaller\smaller\smaller\begin{tabular}{@{}c@{}}%
72\\6\\16
\end{tabular}\endgroup%
\kern3pt%
\begingroup \smaller\smaller\smaller\begin{tabular}{@{}c@{}}%
8\\0\\2
\end{tabular}\endgroup%
\kern3pt%
\begingroup \smaller\smaller\smaller\begin{tabular}{@{}c@{}}%
72\\-6\\16
\end{tabular}\endgroup%
\kern3pt%
\begingroup \smaller\smaller\smaller\begin{tabular}{@{}c@{}}%
36\\-5\\7
\end{tabular}\endgroup%
\kern3pt%
\begingroup \smaller\smaller\smaller\begin{tabular}{@{}c@{}}%
9\\-2\\1
\end{tabular}\endgroup%
\kern3pt%
\begingroup \smaller\smaller\smaller\begin{tabular}{@{}c@{}}%
8\\-2\\0
\end{tabular}\endgroup%
{$\left.\llap{\phantom{%
\begingroup \smaller\smaller\smaller\begin{tabular}{@{}c@{}}%
0\\0\\0
\end{tabular}\endgroup%
}}\!\right]$}%
}%
\ifdim\wd\matricesbox>\halfwidth\myboxwidth=\hsize\else\myboxwidth=\halfwidth\fi
\vbox{%
\ifdim\myboxwidth=\hsize
\setbox\onelinebox=\hbox{%
\vbox{\hbox{%
$\Pi_{18,26}$ spans $L_{142.20}$%
}\hbox{%
$\infty4224\infty42|24\infty4224\infty2|2\rtimes D_{2}$%
}%
}%
\hfill\copy\matricesbox
}%
\ifdim\wd\onelinebox>\myboxwidth
\hbox to \myboxwidth{%
$\Pi_{18,26}$ spans $L_{142.20}$%
\hfil
$\infty4224\infty42|24\infty4224\infty2|2\rtimes D_{2}$%
}%
\box\matricesbox
\else
\hbox to \myboxwidth{%
\unhbox\onelinebox
}%
\fi
\else
\hbox to \myboxwidth{%
$\Pi_{18,26}$ spans $L_{142.20}$%
\hfil}%
\hbox to \myboxwidth{%
$\infty4224\infty42|24\infty4224\infty2|2\rtimes D_{2}$%
\hfil}%
\box\matricesbox
\fi
}%
\hfill\discretionary{}{}{}%
\setbox\matricesbox=\hbox{%
{$\left[\!\llap{\phantom{%
\begingroup \smaller\smaller\smaller\begin{tabular}{@{}c@{}}%
\phantom{0}\\\phantom{0}\\\phantom{0}
\end{tabular}\endgroup%
}}\right.$}%
\begingroup \smaller\smaller\smaller\begin{tabular}{@{}c@{}}%
-1\\\phantom{0}\\\phantom{0}
\end{tabular}\endgroup%
\kern3pt%
\begingroup \smaller\smaller\smaller\begin{tabular}{@{}c@{}}%
\phantom{0}\\9\\\phantom{0}
\end{tabular}\endgroup%
\kern3pt%
\begingroup \smaller\smaller\smaller\begin{tabular}{@{}c@{}}%
\phantom{0}\\\phantom{0}\\9
\end{tabular}\endgroup%
{$\left.\llap{\phantom{%
\begingroup \smaller\smaller\smaller\begin{tabular}{@{}c@{}}%
\phantom{0}\\\phantom{0}\\\phantom{0}
\end{tabular}\endgroup%
}}\!\right]$}%
{$\left[\!\llap{\phantom{%
\begingroup \smaller\smaller\smaller\begin{tabular}{@{}c@{}}%
0\\0\\0
\end{tabular}\endgroup%
}}\right.$}%
\begingroup \smaller\smaller\smaller\begin{tabular}{@{}c@{}}%
9\\3\\-1
\end{tabular}\endgroup%
\kern3pt%
\begingroup \smaller\smaller\smaller\begin{tabular}{@{}c@{}}%
8\\2\\-2
\end{tabular}\endgroup%
\kern3pt%
\begingroup \smaller\smaller\smaller\begin{tabular}{@{}c@{}}%
72\\10\\-22
\end{tabular}\endgroup%
\kern3pt%
\begingroup \smaller\smaller\smaller\begin{tabular}{@{}c@{}}%
36\\2\\-12
\end{tabular}\endgroup%
\kern3pt%
\begingroup \smaller\smaller\smaller\begin{tabular}{@{}c@{}}%
36\\-2\\-12
\end{tabular}\endgroup%
\kern3pt%
\begingroup \smaller\smaller\smaller\begin{tabular}{@{}c@{}}%
72\\-10\\-22
\end{tabular}\endgroup%
\kern3pt%
\begingroup \smaller\smaller\smaller\begin{tabular}{@{}c@{}}%
8\\-2\\-2
\end{tabular}\endgroup%
\kern3pt%
\begingroup \smaller\smaller\smaller\begin{tabular}{@{}c@{}}%
72\\-22\\-10
\end{tabular}\endgroup%
\kern3pt%
\begingroup \smaller\smaller\smaller\begin{tabular}{@{}c@{}}%
36\\-12\\-2
\end{tabular}\endgroup%
{$\left.\llap{\phantom{%
\begingroup \smaller\smaller\smaller\begin{tabular}{@{}c@{}}%
0\\0\\0
\end{tabular}\endgroup%
}}\!\right]$}%
}%
\ifdim\wd\matricesbox>\halfwidth\myboxwidth=\hsize\else\myboxwidth=\halfwidth\fi
\vbox{%
\ifdim\myboxwidth=\hsize
\setbox\onelinebox=\hbox{%
\vbox{\hbox{%
$\Pi_{18,27}$ spans $L_{142.20}$%
}\hbox{%
$422\slashinfty224\infty4224\slashinfty4224\infty\rtimes D_{2}$%
}%
}%
\hfill\copy\matricesbox
}%
\ifdim\wd\onelinebox>\myboxwidth
\hbox to \myboxwidth{%
$\Pi_{18,27}$ spans $L_{142.20}$%
\hfil
$422\slashinfty224\infty4224\slashinfty4224\infty\rtimes D_{2}$%
}%
\box\matricesbox
\else
\hbox to \myboxwidth{%
\unhbox\onelinebox
}%
\fi
\else
\hbox to \myboxwidth{%
$\Pi_{18,27}$ spans $L_{142.20}$%
\hfil}%
\hbox to \myboxwidth{%
$422\slashinfty224\infty4224\slashinfty4224\infty\rtimes D_{2}$%
\hfil}%
\box\matricesbox
\fi
}%
\hfill\discretionary{}{}{}%
\setbox\matricesbox=\hbox{%
{$\left[\!\llap{\phantom{%
\begingroup \smaller\smaller\smaller\begin{tabular}{@{}c@{}}%
\phantom{0}\\\phantom{0}\\\phantom{0}
\end{tabular}\endgroup%
}}\right.$}%
\begingroup \smaller\smaller\smaller\begin{tabular}{@{}c@{}}%
-1\\\phantom{0}\\\phantom{0}
\end{tabular}\endgroup%
\kern3pt%
\begingroup \smaller\smaller\smaller\begin{tabular}{@{}c@{}}%
\phantom{0}\\30\\-15
\end{tabular}\endgroup%
\kern3pt%
\begingroup \smaller\smaller\smaller\begin{tabular}{@{}c@{}}%
\phantom{0}\\-15\\30
\end{tabular}\endgroup%
{$\left.\llap{\phantom{%
\begingroup \smaller\smaller\smaller\begin{tabular}{@{}c@{}}%
\phantom{0}\\\phantom{0}\\\phantom{0}
\end{tabular}\endgroup%
}}\!\right]$}%
{$\left[\!\llap{\phantom{%
\begingroup \smaller\smaller\smaller\begin{tabular}{@{}c@{}}%
0\\0\\0
\end{tabular}\endgroup%
}}\right.$}%
\begingroup \smaller\smaller\smaller\begin{tabular}{@{}c@{}}%
90\\-8\\-19
\end{tabular}\endgroup%
\kern3pt%
\begingroup \smaller\smaller\smaller\begin{tabular}{@{}c@{}}%
90\\-11\\-19
\end{tabular}\endgroup%
\kern3pt%
\begingroup \smaller\smaller\smaller\begin{tabular}{@{}c@{}}%
30\\-5\\-6
\end{tabular}\endgroup%
\kern3pt%
\begingroup \smaller\smaller\smaller\begin{tabular}{@{}c@{}}%
30\\-6\\-5
\end{tabular}\endgroup%
\kern3pt%
\begingroup \smaller\smaller\smaller\begin{tabular}{@{}c@{}}%
9\\-2\\-1
\end{tabular}\endgroup%
\kern3pt%
\begingroup \smaller\smaller\smaller\begin{tabular}{@{}c@{}}%
5\\-1\\0
\end{tabular}\endgroup%
\kern3pt%
\begingroup \smaller\smaller\smaller\begin{tabular}{@{}c@{}}%
90\\-11\\8
\end{tabular}\endgroup%
\kern3pt%
\begingroup \smaller\smaller\smaller\begin{tabular}{@{}c@{}}%
90\\-8\\11
\end{tabular}\endgroup%
\kern3pt%
\begingroup \smaller\smaller\smaller\begin{tabular}{@{}c@{}}%
5\\0\\1
\end{tabular}\endgroup%
{$\left.\llap{\phantom{%
\begingroup \smaller\smaller\smaller\begin{tabular}{@{}c@{}}%
0\\0\\0
\end{tabular}\endgroup%
}}\!\right]$}%
}%
\ifdim\wd\matricesbox>\halfwidth\myboxwidth=\hsize\else\myboxwidth=\halfwidth\fi
\vbox{%
\ifdim\myboxwidth=\hsize
\setbox\onelinebox=\hbox{%
\vbox{\hbox{%
$\Pi_{18,28}$ spans $L_{16.13}$%
}\hbox{%
$363222322363222322\rtimes C_{2}$%
}%
}%
\hfill\copy\matricesbox
}%
\ifdim\wd\onelinebox>\myboxwidth
\hbox to \myboxwidth{%
$\Pi_{18,28}$ spans $L_{16.13}$%
\hfil
$363222322363222322\rtimes C_{2}$%
}%
\box\matricesbox
\else
\hbox to \myboxwidth{%
\unhbox\onelinebox
}%
\fi
\else
\hbox to \myboxwidth{%
$\Pi_{18,28}$ spans $L_{16.13}$%
\hfil}%
\hbox to \myboxwidth{%
$363222322363222322\rtimes C_{2}$%
\hfil}%
\box\matricesbox
\fi
}%
\hfill\discretionary{}{}{}%
\setbox\matricesbox=\hbox{%
{$\left[\!\llap{\phantom{%
\begingroup \smaller\smaller\smaller\begin{tabular}{@{}c@{}}%
\phantom{0}\\\phantom{0}\\\phantom{0}
\end{tabular}\endgroup%
}}\right.$}%
\begingroup \smaller\smaller\smaller\begin{tabular}{@{}c@{}}%
-1\\\phantom{0}\\\phantom{0}
\end{tabular}\endgroup%
\kern3pt%
\begingroup \smaller\smaller\smaller\begin{tabular}{@{}c@{}}%
\phantom{0}\\30\\-15
\end{tabular}\endgroup%
\kern3pt%
\begingroup \smaller\smaller\smaller\begin{tabular}{@{}c@{}}%
\phantom{0}\\-15\\30
\end{tabular}\endgroup%
{$\left.\llap{\phantom{%
\begingroup \smaller\smaller\smaller\begin{tabular}{@{}c@{}}%
\phantom{0}\\\phantom{0}\\\phantom{0}
\end{tabular}\endgroup%
}}\!\right]$}%
{$\left[\!\llap{\phantom{%
\begingroup \smaller\smaller\smaller\begin{tabular}{@{}c@{}}%
0\\0\\0
\end{tabular}\endgroup%
}}\right.$}%
\begingroup \smaller\smaller\smaller\begin{tabular}{@{}c@{}}%
90\\-8\\-19
\end{tabular}\endgroup%
\kern3pt%
\begingroup \smaller\smaller\smaller\begin{tabular}{@{}c@{}}%
90\\-11\\-19
\end{tabular}\endgroup%
\kern3pt%
\begingroup \smaller\smaller\smaller\begin{tabular}{@{}c@{}}%
30\\-5\\-6
\end{tabular}\endgroup%
\kern3pt%
\begingroup \smaller\smaller\smaller\begin{tabular}{@{}c@{}}%
30\\-6\\-5
\end{tabular}\endgroup%
\kern3pt%
\begingroup \smaller\smaller\smaller\begin{tabular}{@{}c@{}}%
9\\-2\\-1
\end{tabular}\endgroup%
\kern3pt%
\begingroup \smaller\smaller\smaller\begin{tabular}{@{}c@{}}%
30\\-6\\-1
\end{tabular}\endgroup%
\kern3pt%
\begingroup \smaller\smaller\smaller\begin{tabular}{@{}c@{}}%
30\\-5\\1
\end{tabular}\endgroup%
\kern3pt%
\begingroup \smaller\smaller\smaller\begin{tabular}{@{}c@{}}%
9\\-1\\1
\end{tabular}\endgroup%
\kern3pt%
\begingroup \smaller\smaller\smaller\begin{tabular}{@{}c@{}}%
5\\0\\1
\end{tabular}\endgroup%
{$\left.\llap{\phantom{%
\begingroup \smaller\smaller\smaller\begin{tabular}{@{}c@{}}%
0\\0\\0
\end{tabular}\endgroup%
}}\!\right]$}%
}%
\ifdim\wd\matricesbox>\halfwidth\myboxwidth=\hsize\else\myboxwidth=\halfwidth\fi
\vbox{%
\ifdim\myboxwidth=\hsize
\setbox\onelinebox=\hbox{%
\vbox{\hbox{%
$\Pi_{18,29}$ spans $L_{16.13}$%
}\hbox{%
$363223222363223222\rtimes C_{2}$%
}%
}%
\hfill\copy\matricesbox
}%
\ifdim\wd\onelinebox>\myboxwidth
\hbox to \myboxwidth{%
$\Pi_{18,29}$ spans $L_{16.13}$%
\hfil
$363223222363223222\rtimes C_{2}$%
}%
\box\matricesbox
\else
\hbox to \myboxwidth{%
\unhbox\onelinebox
}%
\fi
\else
\hbox to \myboxwidth{%
$\Pi_{18,29}$ spans $L_{16.13}$%
\hfil}%
\hbox to \myboxwidth{%
$363223222363223222\rtimes C_{2}$%
\hfil}%
\box\matricesbox
\fi
}%
\hfill\discretionary{}{}{}%
\setbox\matricesbox=\hbox{%
{$\left[\!\llap{\phantom{%
\begingroup \smaller\smaller\smaller\begin{tabular}{@{}c@{}}%
\phantom{0}\\\phantom{0}\\\phantom{0}
\end{tabular}\endgroup%
}}\right.$}%
\begingroup \smaller\smaller\smaller\begin{tabular}{@{}c@{}}%
-1\\\phantom{0}\\\phantom{0}
\end{tabular}\endgroup%
\kern3pt%
\begingroup \smaller\smaller\smaller\begin{tabular}{@{}c@{}}%
\phantom{0}\\18\\\phantom{0}
\end{tabular}\endgroup%
\kern3pt%
\begingroup \smaller\smaller\smaller\begin{tabular}{@{}c@{}}%
\phantom{0}\\\phantom{0}\\18
\end{tabular}\endgroup%
{$\left.\llap{\phantom{%
\begingroup \smaller\smaller\smaller\begin{tabular}{@{}c@{}}%
\phantom{0}\\\phantom{0}\\\phantom{0}
\end{tabular}\endgroup%
}}\!\right]$}%
{$\left[\!\llap{\phantom{%
\begingroup \smaller\smaller\smaller\begin{tabular}{@{}c@{}}%
0\\0\\0
\end{tabular}\endgroup%
}}\right.$}%
\begingroup \smaller\smaller\smaller\begin{tabular}{@{}c@{}}%
9\\1\\2
\end{tabular}\endgroup%
\kern3pt%
\begingroup \smaller\smaller\smaller\begin{tabular}{@{}c@{}}%
36\\7\\5
\end{tabular}\endgroup%
\kern3pt%
\begingroup \smaller\smaller\smaller\begin{tabular}{@{}c@{}}%
72\\16\\6
\end{tabular}\endgroup%
\kern3pt%
\begingroup \smaller\smaller\smaller\begin{tabular}{@{}c@{}}%
8\\2\\0
\end{tabular}\endgroup%
\kern3pt%
\begingroup \smaller\smaller\smaller\begin{tabular}{@{}c@{}}%
72\\16\\-6
\end{tabular}\endgroup%
\kern3pt%
\begingroup \smaller\smaller\smaller\begin{tabular}{@{}c@{}}%
36\\7\\-5
\end{tabular}\endgroup%
\kern3pt%
\begingroup \smaller\smaller\smaller\begin{tabular}{@{}c@{}}%
36\\5\\-7
\end{tabular}\endgroup%
\kern3pt%
\begingroup \smaller\smaller\smaller\begin{tabular}{@{}c@{}}%
72\\6\\-16
\end{tabular}\endgroup%
\kern3pt%
\begingroup \smaller\smaller\smaller\begin{tabular}{@{}c@{}}%
8\\0\\-2
\end{tabular}\endgroup%
{$\left.\llap{\phantom{%
\begingroup \smaller\smaller\smaller\begin{tabular}{@{}c@{}}%
0\\0\\0
\end{tabular}\endgroup%
}}\!\right]$}%
}%
\ifdim\wd\matricesbox>\halfwidth\myboxwidth=\hsize\else\myboxwidth=\halfwidth\fi
\vbox{%
\ifdim\myboxwidth=\hsize
\setbox\onelinebox=\hbox{%
\vbox{\hbox{%
$\Pi_{18,30}$ spans $L_{142.20}$%
}\hbox{%
$\infty4224\infty422\infty4224\infty422\rtimes C_{2}$%
}%
}%
\hfill\copy\matricesbox
}%
\ifdim\wd\onelinebox>\myboxwidth
\hbox to \myboxwidth{%
$\Pi_{18,30}$ spans $L_{142.20}$%
\hfil
$\infty4224\infty422\infty4224\infty422\rtimes C_{2}$%
}%
\box\matricesbox
\else
\hbox to \myboxwidth{%
\unhbox\onelinebox
}%
\fi
\else
\hbox to \myboxwidth{%
$\Pi_{18,30}$ spans $L_{142.20}$%
\hfil}%
\hbox to \myboxwidth{%
$\infty4224\infty422\infty4224\infty422\rtimes C_{2}$%
\hfil}%
\box\matricesbox
\fi
}%
\hfill\discretionary{}{}{}%
\setbox\matricesbox=\hbox{%
{$\left[\!\llap{\phantom{%
\begingroup \smaller\smaller\smaller
\endgroup%
}}\!\right]$}%
}%
\ifdim\wd\matricesbox>\halfwidth\myboxwidth=\hsize\else\myboxwidth=\halfwidth\fi
\vbox{%
\ifdim\myboxwidth=\hsize
\setbox\onelinebox=\hbox{%
\vbox{\hbox{%
$\Pi_{18,31}$ spans $L_{16.13}$%
}\hbox{%
$222222363636363632$%
}%
}%
\hfill\copy\matricesbox
}%
\ifdim\wd\onelinebox>\myboxwidth
\hbox to \myboxwidth{%
$\Pi_{18,31}$ spans $L_{16.13}$%
\hfil
$222222363636363632$%
}%
\box\matricesbox
\else
\hbox to \myboxwidth{%
\unhbox\onelinebox
}%
\fi
\else
\hbox to \myboxwidth{%
$\Pi_{18,31}$ spans $L_{16.13}$%
\hfil}%
\hbox to \myboxwidth{%
$222222363636363632$%
\hfil}%
\box\matricesbox
\fi
}%
\hfill\discretionary{}{}{}%
\setbox\matricesbox=\hbox{%
{$\left[\!\llap{\phantom{%
\begingroup \smaller\smaller\smaller
\endgroup%
}}\!\right]$}%
}%
\ifdim\wd\matricesbox>\halfwidth\myboxwidth=\hsize\else\myboxwidth=\halfwidth\fi
\vbox{%
\ifdim\myboxwidth=\hsize
\setbox\onelinebox=\hbox{%
\vbox{\hbox{%
$\Pi_{18,32}$ spans $L_{16.13}$%
}\hbox{%
$222223223636363632$%
}%
}%
\hfill\copy\matricesbox
}%
\ifdim\wd\onelinebox>\myboxwidth
\hbox to \myboxwidth{%
$\Pi_{18,32}$ spans $L_{16.13}$%
\hfil
$222223223636363632$%
}%
\box\matricesbox
\else
\hbox to \myboxwidth{%
\unhbox\onelinebox
}%
\fi
\else
\hbox to \myboxwidth{%
$\Pi_{18,32}$ spans $L_{16.13}$%
\hfil}%
\hbox to \myboxwidth{%
$222223223636363632$%
\hfil}%
\box\matricesbox
\fi
}%
\hfill\discretionary{}{}{}%
\setbox\matricesbox=\hbox{%
{$\left[\!\llap{\phantom{%
\begingroup \smaller\smaller\smaller
\endgroup%
}}\!\right]$}%
}%
\ifdim\wd\matricesbox>\halfwidth\myboxwidth=\hsize\else\myboxwidth=\halfwidth\fi
\vbox{%
\ifdim\myboxwidth=\hsize
\setbox\onelinebox=\hbox{%
\vbox{\hbox{%
$\Pi_{18,33}$ spans $L_{16.13}$%
}\hbox{%
$222223632236363632$%
}%
}%
\hfill\copy\matricesbox
}%
\ifdim\wd\onelinebox>\myboxwidth
\hbox to \myboxwidth{%
$\Pi_{18,33}$ spans $L_{16.13}$%
\hfil
$222223632236363632$%
}%
\box\matricesbox
\else
\hbox to \myboxwidth{%
\unhbox\onelinebox
}%
\fi
\else
\hbox to \myboxwidth{%
$\Pi_{18,33}$ spans $L_{16.13}$%
\hfil}%
\hbox to \myboxwidth{%
$222223632236363632$%
\hfil}%
\box\matricesbox
\fi
}%
\hfill\discretionary{}{}{}%
\setbox\matricesbox=\hbox{%
{$\left[\!\llap{\phantom{%
\begingroup \smaller\smaller\smaller
\endgroup%
}}\!\right]$}%
}%
\ifdim\wd\matricesbox>\halfwidth\myboxwidth=\hsize\else\myboxwidth=\halfwidth\fi
\vbox{%
\ifdim\myboxwidth=\hsize
\setbox\onelinebox=\hbox{%
\vbox{\hbox{%
$\Pi_{18,34}$ spans $L_{16.13}$%
}\hbox{%
$363222222322363636$%
}%
}%
\hfill\copy\matricesbox
}%
\ifdim\wd\onelinebox>\myboxwidth
\hbox to \myboxwidth{%
$\Pi_{18,34}$ spans $L_{16.13}$%
\hfil
$363222222322363636$%
}%
\box\matricesbox
\else
\hbox to \myboxwidth{%
\unhbox\onelinebox
}%
\fi
\else
\hbox to \myboxwidth{%
$\Pi_{18,34}$ spans $L_{16.13}$%
\hfil}%
\hbox to \myboxwidth{%
$363222222322363636$%
\hfil}%
\box\matricesbox
\fi
}%
\hfill\discretionary{}{}{}%
\setbox\matricesbox=\hbox{%
{$\left[\!\llap{\phantom{%
\begingroup \smaller\smaller\smaller
\endgroup%
}}\!\right]$}%
}%
\ifdim\wd\matricesbox>\halfwidth\myboxwidth=\hsize\else\myboxwidth=\halfwidth\fi
\vbox{%
\ifdim\myboxwidth=\hsize
\setbox\onelinebox=\hbox{%
\vbox{\hbox{%
$\Pi_{18,35}$ spans $L_{16.13}$%
}\hbox{%
$363222222363223636$%
}%
}%
\hfill\copy\matricesbox
}%
\ifdim\wd\onelinebox>\myboxwidth
\hbox to \myboxwidth{%
$\Pi_{18,35}$ spans $L_{16.13}$%
\hfil
$363222222363223636$%
}%
\box\matricesbox
\else
\hbox to \myboxwidth{%
\unhbox\onelinebox
}%
\fi
\else
\hbox to \myboxwidth{%
$\Pi_{18,35}$ spans $L_{16.13}$%
\hfil}%
\hbox to \myboxwidth{%
$363222222363223636$%
\hfil}%
\box\matricesbox
\fi
}%
\hfill\discretionary{}{}{}%
\setbox\matricesbox=\hbox{%
{$\left[\!\llap{\phantom{%
\begingroup \smaller\smaller\smaller
\endgroup%
}}\!\right]$}%
}%
\ifdim\wd\matricesbox>\halfwidth\myboxwidth=\hsize\else\myboxwidth=\halfwidth\fi
\vbox{%
\ifdim\myboxwidth=\hsize
\setbox\onelinebox=\hbox{%
\vbox{\hbox{%
$\Pi_{18,36}$ spans $L_{16.13}$%
}\hbox{%
$363222223222363636$%
}%
}%
\hfill\copy\matricesbox
}%
\ifdim\wd\onelinebox>\myboxwidth
\hbox to \myboxwidth{%
$\Pi_{18,36}$ spans $L_{16.13}$%
\hfil
$363222223222363636$%
}%
\box\matricesbox
\else
\hbox to \myboxwidth{%
\unhbox\onelinebox
}%
\fi
\else
\hbox to \myboxwidth{%
$\Pi_{18,36}$ spans $L_{16.13}$%
\hfil}%
\hbox to \myboxwidth{%
$363222223222363636$%
\hfil}%
\box\matricesbox
\fi
}%
\hfill\discretionary{}{}{}%
\setbox\matricesbox=\hbox{%
{$\left[\!\llap{\phantom{%
\begingroup \smaller\smaller\smaller
\endgroup%
}}\!\right]$}%
}%
\ifdim\wd\matricesbox>\halfwidth\myboxwidth=\hsize\else\myboxwidth=\halfwidth\fi
\vbox{%
\ifdim\myboxwidth=\hsize
\setbox\onelinebox=\hbox{%
\vbox{\hbox{%
$\Pi_{18,37}$ spans $L_{16.13}$%
}\hbox{%
$363222223223223636$%
}%
}%
\hfill\copy\matricesbox
}%
\ifdim\wd\onelinebox>\myboxwidth
\hbox to \myboxwidth{%
$\Pi_{18,37}$ spans $L_{16.13}$%
\hfil
$363222223223223636$%
}%
\box\matricesbox
\else
\hbox to \myboxwidth{%
\unhbox\onelinebox
}%
\fi
\else
\hbox to \myboxwidth{%
$\Pi_{18,37}$ spans $L_{16.13}$%
\hfil}%
\hbox to \myboxwidth{%
$363222223223223636$%
\hfil}%
\box\matricesbox
\fi
}%
\hfill\discretionary{}{}{}%
\setbox\matricesbox=\hbox{%
{$\left[\!\llap{\phantom{%
\begingroup \smaller\smaller\smaller
\endgroup%
}}\!\right]$}%
}%
\ifdim\wd\matricesbox>\halfwidth\myboxwidth=\hsize\else\myboxwidth=\halfwidth\fi
\vbox{%
\ifdim\myboxwidth=\hsize
\setbox\onelinebox=\hbox{%
\vbox{\hbox{%
$\Pi_{18,38}$ spans $L_{16.13}$%
}\hbox{%
$363222223223632236$%
}%
}%
\hfill\copy\matricesbox
}%
\ifdim\wd\onelinebox>\myboxwidth
\hbox to \myboxwidth{%
$\Pi_{18,38}$ spans $L_{16.13}$%
\hfil
$363222223223632236$%
}%
\box\matricesbox
\else
\hbox to \myboxwidth{%
\unhbox\onelinebox
}%
\fi
\else
\hbox to \myboxwidth{%
$\Pi_{18,38}$ spans $L_{16.13}$%
\hfil}%
\hbox to \myboxwidth{%
$363222223223632236$%
\hfil}%
\box\matricesbox
\fi
}%
\hfill\discretionary{}{}{}%
\setbox\matricesbox=\hbox{%
{$\left[\!\llap{\phantom{%
\begingroup \smaller\smaller\smaller
\endgroup%
}}\!\right]$}%
}%
\ifdim\wd\matricesbox>\halfwidth\myboxwidth=\hsize\else\myboxwidth=\halfwidth\fi
\vbox{%
\ifdim\myboxwidth=\hsize
\setbox\onelinebox=\hbox{%
\vbox{\hbox{%
$\Pi_{18,39}$ spans $L_{16.13}$%
}\hbox{%
$363222223223636322$%
}%
}%
\hfill\copy\matricesbox
}%
\ifdim\wd\onelinebox>\myboxwidth
\hbox to \myboxwidth{%
$\Pi_{18,39}$ spans $L_{16.13}$%
\hfil
$363222223223636322$%
}%
\box\matricesbox
\else
\hbox to \myboxwidth{%
\unhbox\onelinebox
}%
\fi
\else
\hbox to \myboxwidth{%
$\Pi_{18,39}$ spans $L_{16.13}$%
\hfil}%
\hbox to \myboxwidth{%
$363222223223636322$%
\hfil}%
\box\matricesbox
\fi
}%
\hfill\discretionary{}{}{}%
\setbox\matricesbox=\hbox{%
{$\left[\!\llap{\phantom{%
\begingroup \smaller\smaller\smaller
\endgroup%
}}\!\right]$}%
}%
\ifdim\wd\matricesbox>\halfwidth\myboxwidth=\hsize\else\myboxwidth=\halfwidth\fi
\vbox{%
\ifdim\myboxwidth=\hsize
\setbox\onelinebox=\hbox{%
\vbox{\hbox{%
$\Pi_{18,40}$ spans $L_{16.13}$%
}\hbox{%
$363222223632223636$%
}%
}%
\hfill\copy\matricesbox
}%
\ifdim\wd\onelinebox>\myboxwidth
\hbox to \myboxwidth{%
$\Pi_{18,40}$ spans $L_{16.13}$%
\hfil
$363222223632223636$%
}%
\box\matricesbox
\else
\hbox to \myboxwidth{%
\unhbox\onelinebox
}%
\fi
\else
\hbox to \myboxwidth{%
$\Pi_{18,40}$ spans $L_{16.13}$%
\hfil}%
\hbox to \myboxwidth{%
$363222223632223636$%
\hfil}%
\box\matricesbox
\fi
}%
\hfill\discretionary{}{}{}%
\setbox\matricesbox=\hbox{%
{$\left[\!\llap{\phantom{%
\begingroup \smaller\smaller\smaller
\endgroup%
}}\!\right]$}%
}%
\ifdim\wd\matricesbox>\halfwidth\myboxwidth=\hsize\else\myboxwidth=\halfwidth\fi
\vbox{%
\ifdim\myboxwidth=\hsize
\setbox\onelinebox=\hbox{%
\vbox{\hbox{%
$\Pi_{18,41}$ spans $L_{16.13}$%
}\hbox{%
$363222223632232236$%
}%
}%
\hfill\copy\matricesbox
}%
\ifdim\wd\onelinebox>\myboxwidth
\hbox to \myboxwidth{%
$\Pi_{18,41}$ spans $L_{16.13}$%
\hfil
$363222223632232236$%
}%
\box\matricesbox
\else
\hbox to \myboxwidth{%
\unhbox\onelinebox
}%
\fi
\else
\hbox to \myboxwidth{%
$\Pi_{18,41}$ spans $L_{16.13}$%
\hfil}%
\hbox to \myboxwidth{%
$363222223632232236$%
\hfil}%
\box\matricesbox
\fi
}%
\hfill\discretionary{}{}{}%
\setbox\matricesbox=\hbox{%
{$\left[\!\llap{\phantom{%
\begingroup \smaller\smaller\smaller
\endgroup%
}}\!\right]$}%
}%
\ifdim\wd\matricesbox>\halfwidth\myboxwidth=\hsize\else\myboxwidth=\halfwidth\fi
\vbox{%
\ifdim\myboxwidth=\hsize
\setbox\onelinebox=\hbox{%
\vbox{\hbox{%
$\Pi_{18,42}$ spans $L_{16.13}$%
}\hbox{%
$363222223632236322$%
}%
}%
\hfill\copy\matricesbox
}%
\ifdim\wd\onelinebox>\myboxwidth
\hbox to \myboxwidth{%
$\Pi_{18,42}$ spans $L_{16.13}$%
\hfil
$363222223632236322$%
}%
\box\matricesbox
\else
\hbox to \myboxwidth{%
\unhbox\onelinebox
}%
\fi
\else
\hbox to \myboxwidth{%
$\Pi_{18,42}$ spans $L_{16.13}$%
\hfil}%
\hbox to \myboxwidth{%
$363222223632236322$%
\hfil}%
\box\matricesbox
\fi
}%
\hfill\discretionary{}{}{}%
\setbox\matricesbox=\hbox{%
{$\left[\!\llap{\phantom{%
\begingroup \smaller\smaller\smaller
\endgroup%
}}\!\right]$}%
}%
\ifdim\wd\matricesbox>\halfwidth\myboxwidth=\hsize\else\myboxwidth=\halfwidth\fi
\vbox{%
\ifdim\myboxwidth=\hsize
\setbox\onelinebox=\hbox{%
\vbox{\hbox{%
$\Pi_{18,43}$ spans $L_{16.13}$%
}\hbox{%
$363222223636322236$%
}%
}%
\hfill\copy\matricesbox
}%
\ifdim\wd\onelinebox>\myboxwidth
\hbox to \myboxwidth{%
$\Pi_{18,43}$ spans $L_{16.13}$%
\hfil
$363222223636322236$%
}%
\box\matricesbox
\else
\hbox to \myboxwidth{%
\unhbox\onelinebox
}%
\fi
\else
\hbox to \myboxwidth{%
$\Pi_{18,43}$ spans $L_{16.13}$%
\hfil}%
\hbox to \myboxwidth{%
$363222223636322236$%
\hfil}%
\box\matricesbox
\fi
}%
\hfill\discretionary{}{}{}%
\setbox\matricesbox=\hbox{%
{$\left[\!\llap{\phantom{%
\begingroup \smaller\smaller\smaller
\endgroup%
}}\!\right]$}%
}%
\ifdim\wd\matricesbox>\halfwidth\myboxwidth=\hsize\else\myboxwidth=\halfwidth\fi
\vbox{%
\ifdim\myboxwidth=\hsize
\setbox\onelinebox=\hbox{%
\vbox{\hbox{%
$\Pi_{18,44}$ spans $L_{16.13}$%
}\hbox{%
$363222223636322322$%
}%
}%
\hfill\copy\matricesbox
}%
\ifdim\wd\onelinebox>\myboxwidth
\hbox to \myboxwidth{%
$\Pi_{18,44}$ spans $L_{16.13}$%
\hfil
$363222223636322322$%
}%
\box\matricesbox
\else
\hbox to \myboxwidth{%
\unhbox\onelinebox
}%
\fi
\else
\hbox to \myboxwidth{%
$\Pi_{18,44}$ spans $L_{16.13}$%
\hfil}%
\hbox to \myboxwidth{%
$363222223636322322$%
\hfil}%
\box\matricesbox
\fi
}%
\hfill\discretionary{}{}{}%
\setbox\matricesbox=\hbox{%
{$\left[\!\llap{\phantom{%
\begingroup \smaller\smaller\smaller
\endgroup%
}}\!\right]$}%
}%
\ifdim\wd\matricesbox>\halfwidth\myboxwidth=\hsize\else\myboxwidth=\halfwidth\fi
\vbox{%
\ifdim\myboxwidth=\hsize
\setbox\onelinebox=\hbox{%
\vbox{\hbox{%
$\Pi_{18,45}$ spans $L_{16.13}$%
}\hbox{%
$363222223636363222$%
}%
}%
\hfill\copy\matricesbox
}%
\ifdim\wd\onelinebox>\myboxwidth
\hbox to \myboxwidth{%
$\Pi_{18,45}$ spans $L_{16.13}$%
\hfil
$363222223636363222$%
}%
\box\matricesbox
\else
\hbox to \myboxwidth{%
\unhbox\onelinebox
}%
\fi
\else
\hbox to \myboxwidth{%
$\Pi_{18,45}$ spans $L_{16.13}$%
\hfil}%
\hbox to \myboxwidth{%
$363222223636363222$%
\hfil}%
\box\matricesbox
\fi
}%
\hfill\discretionary{}{}{}%
\setbox\matricesbox=\hbox{%
{$\left[\!\llap{\phantom{%
\begingroup \smaller\smaller\smaller
\endgroup%
}}\!\right]$}%
}%
\ifdim\wd\matricesbox>\halfwidth\myboxwidth=\hsize\else\myboxwidth=\halfwidth\fi
\vbox{%
\ifdim\myboxwidth=\hsize
\setbox\onelinebox=\hbox{%
\vbox{\hbox{%
$\Pi_{18,46}$ spans $L_{16.13}$%
}\hbox{%
$363222232223223636$%
}%
}%
\hfill\copy\matricesbox
}%
\ifdim\wd\onelinebox>\myboxwidth
\hbox to \myboxwidth{%
$\Pi_{18,46}$ spans $L_{16.13}$%
\hfil
$363222232223223636$%
}%
\box\matricesbox
\else
\hbox to \myboxwidth{%
\unhbox\onelinebox
}%
\fi
\else
\hbox to \myboxwidth{%
$\Pi_{18,46}$ spans $L_{16.13}$%
\hfil}%
\hbox to \myboxwidth{%
$363222232223223636$%
\hfil}%
\box\matricesbox
\fi
}%
\hfill\discretionary{}{}{}%
\setbox\matricesbox=\hbox{%
{$\left[\!\llap{\phantom{%
\begingroup \smaller\smaller\smaller
\endgroup%
}}\!\right]$}%
}%
\ifdim\wd\matricesbox>\halfwidth\myboxwidth=\hsize\else\myboxwidth=\halfwidth\fi
\vbox{%
\ifdim\myboxwidth=\hsize
\setbox\onelinebox=\hbox{%
\vbox{\hbox{%
$\Pi_{18,47}$ spans $L_{16.13}$%
}\hbox{%
$363222232223632236$%
}%
}%
\hfill\copy\matricesbox
}%
\ifdim\wd\onelinebox>\myboxwidth
\hbox to \myboxwidth{%
$\Pi_{18,47}$ spans $L_{16.13}$%
\hfil
$363222232223632236$%
}%
\box\matricesbox
\else
\hbox to \myboxwidth{%
\unhbox\onelinebox
}%
\fi
\else
\hbox to \myboxwidth{%
$\Pi_{18,47}$ spans $L_{16.13}$%
\hfil}%
\hbox to \myboxwidth{%
$363222232223632236$%
\hfil}%
\box\matricesbox
\fi
}%
\hfill\discretionary{}{}{}%
\setbox\matricesbox=\hbox{%
{$\left[\!\llap{\phantom{%
\begingroup \smaller\smaller\smaller
\endgroup%
}}\!\right]$}%
}%
\ifdim\wd\matricesbox>\halfwidth\myboxwidth=\hsize\else\myboxwidth=\halfwidth\fi
\vbox{%
\ifdim\myboxwidth=\hsize
\setbox\onelinebox=\hbox{%
\vbox{\hbox{%
$\Pi_{18,48}$ spans $L_{16.13}$%
}\hbox{%
$363222232223636322$%
}%
}%
\hfill\copy\matricesbox
}%
\ifdim\wd\onelinebox>\myboxwidth
\hbox to \myboxwidth{%
$\Pi_{18,48}$ spans $L_{16.13}$%
\hfil
$363222232223636322$%
}%
\box\matricesbox
\else
\hbox to \myboxwidth{%
\unhbox\onelinebox
}%
\fi
\else
\hbox to \myboxwidth{%
$\Pi_{18,48}$ spans $L_{16.13}$%
\hfil}%
\hbox to \myboxwidth{%
$363222232223636322$%
\hfil}%
\box\matricesbox
\fi
}%
\hfill\discretionary{}{}{}%
\setbox\matricesbox=\hbox{%
{$\left[\!\llap{\phantom{%
\begingroup \smaller\smaller\smaller
\endgroup%
}}\!\right]$}%
}%
\ifdim\wd\matricesbox>\halfwidth\myboxwidth=\hsize\else\myboxwidth=\halfwidth\fi
\vbox{%
\ifdim\myboxwidth=\hsize
\setbox\onelinebox=\hbox{%
\vbox{\hbox{%
$\Pi_{18,49}$ spans $L_{16.13}$%
}\hbox{%
$363222232232223636$%
}%
}%
\hfill\copy\matricesbox
}%
\ifdim\wd\onelinebox>\myboxwidth
\hbox to \myboxwidth{%
$\Pi_{18,49}$ spans $L_{16.13}$%
\hfil
$363222232232223636$%
}%
\box\matricesbox
\else
\hbox to \myboxwidth{%
\unhbox\onelinebox
}%
\fi
\else
\hbox to \myboxwidth{%
$\Pi_{18,49}$ spans $L_{16.13}$%
\hfil}%
\hbox to \myboxwidth{%
$363222232232223636$%
\hfil}%
\box\matricesbox
\fi
}%
\hfill\discretionary{}{}{}%
\setbox\matricesbox=\hbox{%
{$\left[\!\llap{\phantom{%
\begingroup \smaller\smaller\smaller
\endgroup%
}}\!\right]$}%
}%
\ifdim\wd\matricesbox>\halfwidth\myboxwidth=\hsize\else\myboxwidth=\halfwidth\fi
\vbox{%
\ifdim\myboxwidth=\hsize
\setbox\onelinebox=\hbox{%
\vbox{\hbox{%
$\Pi_{18,50}$ spans $L_{16.13}$%
}\hbox{%
$363222232232232236$%
}%
}%
\hfill\copy\matricesbox
}%
\ifdim\wd\onelinebox>\myboxwidth
\hbox to \myboxwidth{%
$\Pi_{18,50}$ spans $L_{16.13}$%
\hfil
$363222232232232236$%
}%
\box\matricesbox
\else
\hbox to \myboxwidth{%
\unhbox\onelinebox
}%
\fi
\else
\hbox to \myboxwidth{%
$\Pi_{18,50}$ spans $L_{16.13}$%
\hfil}%
\hbox to \myboxwidth{%
$363222232232232236$%
\hfil}%
\box\matricesbox
\fi
}%
\hfill\discretionary{}{}{}%
\setbox\matricesbox=\hbox{%
{$\left[\!\llap{\phantom{%
\begingroup \smaller\smaller\smaller
\endgroup%
}}\!\right]$}%
}%
\ifdim\wd\matricesbox>\halfwidth\myboxwidth=\hsize\else\myboxwidth=\halfwidth\fi
\vbox{%
\ifdim\myboxwidth=\hsize
\setbox\onelinebox=\hbox{%
\vbox{\hbox{%
$\Pi_{18,51}$ spans $L_{16.13}$%
}\hbox{%
$363222232232236322$%
}%
}%
\hfill\copy\matricesbox
}%
\ifdim\wd\onelinebox>\myboxwidth
\hbox to \myboxwidth{%
$\Pi_{18,51}$ spans $L_{16.13}$%
\hfil
$363222232232236322$%
}%
\box\matricesbox
\else
\hbox to \myboxwidth{%
\unhbox\onelinebox
}%
\fi
\else
\hbox to \myboxwidth{%
$\Pi_{18,51}$ spans $L_{16.13}$%
\hfil}%
\hbox to \myboxwidth{%
$363222232232236322$%
\hfil}%
\box\matricesbox
\fi
}%
\hfill\discretionary{}{}{}%
\setbox\matricesbox=\hbox{%
{$\left[\!\llap{\phantom{%
\begingroup \smaller\smaller\smaller
\endgroup%
}}\!\right]$}%
}%
\ifdim\wd\matricesbox>\halfwidth\myboxwidth=\hsize\else\myboxwidth=\halfwidth\fi
\vbox{%
\ifdim\myboxwidth=\hsize
\setbox\onelinebox=\hbox{%
\vbox{\hbox{%
$\Pi_{18,52}$ spans $L_{16.13}$%
}\hbox{%
$363222232236322236$%
}%
}%
\hfill\copy\matricesbox
}%
\ifdim\wd\onelinebox>\myboxwidth
\hbox to \myboxwidth{%
$\Pi_{18,52}$ spans $L_{16.13}$%
\hfil
$363222232236322236$%
}%
\box\matricesbox
\else
\hbox to \myboxwidth{%
\unhbox\onelinebox
}%
\fi
\else
\hbox to \myboxwidth{%
$\Pi_{18,52}$ spans $L_{16.13}$%
\hfil}%
\hbox to \myboxwidth{%
$363222232236322236$%
\hfil}%
\box\matricesbox
\fi
}%
\hfill\discretionary{}{}{}%
\setbox\matricesbox=\hbox{%
{$\left[\!\llap{\phantom{%
\begingroup \smaller\smaller\smaller
\endgroup%
}}\!\right]$}%
}%
\ifdim\wd\matricesbox>\halfwidth\myboxwidth=\hsize\else\myboxwidth=\halfwidth\fi
\vbox{%
\ifdim\myboxwidth=\hsize
\setbox\onelinebox=\hbox{%
\vbox{\hbox{%
$\Pi_{18,53}$ spans $L_{16.13}$%
}\hbox{%
$363222232236322322$%
}%
}%
\hfill\copy\matricesbox
}%
\ifdim\wd\onelinebox>\myboxwidth
\hbox to \myboxwidth{%
$\Pi_{18,53}$ spans $L_{16.13}$%
\hfil
$363222232236322322$%
}%
\box\matricesbox
\else
\hbox to \myboxwidth{%
\unhbox\onelinebox
}%
\fi
\else
\hbox to \myboxwidth{%
$\Pi_{18,53}$ spans $L_{16.13}$%
\hfil}%
\hbox to \myboxwidth{%
$363222232236322322$%
\hfil}%
\box\matricesbox
\fi
}%
\hfill\discretionary{}{}{}%
\setbox\matricesbox=\hbox{%
{$\left[\!\llap{\phantom{%
\begingroup \smaller\smaller\smaller
\endgroup%
}}\!\right]$}%
}%
\ifdim\wd\matricesbox>\halfwidth\myboxwidth=\hsize\else\myboxwidth=\halfwidth\fi
\vbox{%
\ifdim\myboxwidth=\hsize
\setbox\onelinebox=\hbox{%
\vbox{\hbox{%
$\Pi_{18,54}$ spans $L_{16.13}$%
}\hbox{%
$363222232236363222$%
}%
}%
\hfill\copy\matricesbox
}%
\ifdim\wd\onelinebox>\myboxwidth
\hbox to \myboxwidth{%
$\Pi_{18,54}$ spans $L_{16.13}$%
\hfil
$363222232236363222$%
}%
\box\matricesbox
\else
\hbox to \myboxwidth{%
\unhbox\onelinebox
}%
\fi
\else
\hbox to \myboxwidth{%
$\Pi_{18,54}$ spans $L_{16.13}$%
\hfil}%
\hbox to \myboxwidth{%
$363222232236363222$%
\hfil}%
\box\matricesbox
\fi
}%
\hfill\discretionary{}{}{}%
\setbox\matricesbox=\hbox{%
{$\left[\!\llap{\phantom{%
\begingroup \smaller\smaller\smaller
\endgroup%
}}\!\right]$}%
}%
\ifdim\wd\matricesbox>\halfwidth\myboxwidth=\hsize\else\myboxwidth=\halfwidth\fi
\vbox{%
\ifdim\myboxwidth=\hsize
\setbox\onelinebox=\hbox{%
\vbox{\hbox{%
$\Pi_{18,55}$ spans $L_{16.13}$%
}\hbox{%
$363222236322223636$%
}%
}%
\hfill\copy\matricesbox
}%
\ifdim\wd\onelinebox>\myboxwidth
\hbox to \myboxwidth{%
$\Pi_{18,55}$ spans $L_{16.13}$%
\hfil
$363222236322223636$%
}%
\box\matricesbox
\else
\hbox to \myboxwidth{%
\unhbox\onelinebox
}%
\fi
\else
\hbox to \myboxwidth{%
$\Pi_{18,55}$ spans $L_{16.13}$%
\hfil}%
\hbox to \myboxwidth{%
$363222236322223636$%
\hfil}%
\box\matricesbox
\fi
}%
\hfill\discretionary{}{}{}%
\setbox\matricesbox=\hbox{%
{$\left[\!\llap{\phantom{%
\begingroup \smaller\smaller\smaller
\endgroup%
}}\!\right]$}%
}%
\ifdim\wd\matricesbox>\halfwidth\myboxwidth=\hsize\else\myboxwidth=\halfwidth\fi
\vbox{%
\ifdim\myboxwidth=\hsize
\setbox\onelinebox=\hbox{%
\vbox{\hbox{%
$\Pi_{18,56}$ spans $L_{16.13}$%
}\hbox{%
$363222236322232236$%
}%
}%
\hfill\copy\matricesbox
}%
\ifdim\wd\onelinebox>\myboxwidth
\hbox to \myboxwidth{%
$\Pi_{18,56}$ spans $L_{16.13}$%
\hfil
$363222236322232236$%
}%
\box\matricesbox
\else
\hbox to \myboxwidth{%
\unhbox\onelinebox
}%
\fi
\else
\hbox to \myboxwidth{%
$\Pi_{18,56}$ spans $L_{16.13}$%
\hfil}%
\hbox to \myboxwidth{%
$363222236322232236$%
\hfil}%
\box\matricesbox
\fi
}%
\hfill\discretionary{}{}{}%
\setbox\matricesbox=\hbox{%
{$\left[\!\llap{\phantom{%
\begingroup \smaller\smaller\smaller
\endgroup%
}}\!\right]$}%
}%
\ifdim\wd\matricesbox>\halfwidth\myboxwidth=\hsize\else\myboxwidth=\halfwidth\fi
\vbox{%
\ifdim\myboxwidth=\hsize
\setbox\onelinebox=\hbox{%
\vbox{\hbox{%
$\Pi_{18,57}$ spans $L_{16.13}$%
}\hbox{%
$363222236322236322$%
}%
}%
\hfill\copy\matricesbox
}%
\ifdim\wd\onelinebox>\myboxwidth
\hbox to \myboxwidth{%
$\Pi_{18,57}$ spans $L_{16.13}$%
\hfil
$363222236322236322$%
}%
\box\matricesbox
\else
\hbox to \myboxwidth{%
\unhbox\onelinebox
}%
\fi
\else
\hbox to \myboxwidth{%
$\Pi_{18,57}$ spans $L_{16.13}$%
\hfil}%
\hbox to \myboxwidth{%
$363222236322236322$%
\hfil}%
\box\matricesbox
\fi
}%
\hfill\discretionary{}{}{}%
\setbox\matricesbox=\hbox{%
{$\left[\!\llap{\phantom{%
\begingroup \smaller\smaller\smaller
\endgroup%
}}\!\right]$}%
}%
\ifdim\wd\matricesbox>\halfwidth\myboxwidth=\hsize\else\myboxwidth=\halfwidth\fi
\vbox{%
\ifdim\myboxwidth=\hsize
\setbox\onelinebox=\hbox{%
\vbox{\hbox{%
$\Pi_{18,58}$ spans $L_{16.13}$%
}\hbox{%
$363222236322322236$%
}%
}%
\hfill\copy\matricesbox
}%
\ifdim\wd\onelinebox>\myboxwidth
\hbox to \myboxwidth{%
$\Pi_{18,58}$ spans $L_{16.13}$%
\hfil
$363222236322322236$%
}%
\box\matricesbox
\else
\hbox to \myboxwidth{%
\unhbox\onelinebox
}%
\fi
\else
\hbox to \myboxwidth{%
$\Pi_{18,58}$ spans $L_{16.13}$%
\hfil}%
\hbox to \myboxwidth{%
$363222236322322236$%
\hfil}%
\box\matricesbox
\fi
}%
\hfill\discretionary{}{}{}%
\setbox\matricesbox=\hbox{%
{$\left[\!\llap{\phantom{%
\begingroup \smaller\smaller\smaller
\endgroup%
}}\!\right]$}%
}%
\ifdim\wd\matricesbox>\halfwidth\myboxwidth=\hsize\else\myboxwidth=\halfwidth\fi
\vbox{%
\ifdim\myboxwidth=\hsize
\setbox\onelinebox=\hbox{%
\vbox{\hbox{%
$\Pi_{18,59}$ spans $L_{16.13}$%
}\hbox{%
$363222322223223636$%
}%
}%
\hfill\copy\matricesbox
}%
\ifdim\wd\onelinebox>\myboxwidth
\hbox to \myboxwidth{%
$\Pi_{18,59}$ spans $L_{16.13}$%
\hfil
$363222322223223636$%
}%
\box\matricesbox
\else
\hbox to \myboxwidth{%
\unhbox\onelinebox
}%
\fi
\else
\hbox to \myboxwidth{%
$\Pi_{18,59}$ spans $L_{16.13}$%
\hfil}%
\hbox to \myboxwidth{%
$363222322223223636$%
\hfil}%
\box\matricesbox
\fi
}%
\hfill\discretionary{}{}{}%
\setbox\matricesbox=\hbox{%
{$\left[\!\llap{\phantom{%
\begingroup \smaller\smaller\smaller
\endgroup%
}}\!\right]$}%
}%
\ifdim\wd\matricesbox>\halfwidth\myboxwidth=\hsize\else\myboxwidth=\halfwidth\fi
\vbox{%
\ifdim\myboxwidth=\hsize
\setbox\onelinebox=\hbox{%
\vbox{\hbox{%
$\Pi_{18,60}$ spans $L_{16.13}$%
}\hbox{%
$363222322223632236$%
}%
}%
\hfill\copy\matricesbox
}%
\ifdim\wd\onelinebox>\myboxwidth
\hbox to \myboxwidth{%
$\Pi_{18,60}$ spans $L_{16.13}$%
\hfil
$363222322223632236$%
}%
\box\matricesbox
\else
\hbox to \myboxwidth{%
\unhbox\onelinebox
}%
\fi
\else
\hbox to \myboxwidth{%
$\Pi_{18,60}$ spans $L_{16.13}$%
\hfil}%
\hbox to \myboxwidth{%
$363222322223632236$%
\hfil}%
\box\matricesbox
\fi
}%
\hfill\discretionary{}{}{}%
\setbox\matricesbox=\hbox{%
{$\left[\!\llap{\phantom{%
\begingroup \smaller\smaller\smaller
\endgroup%
}}\!\right]$}%
}%
\ifdim\wd\matricesbox>\halfwidth\myboxwidth=\hsize\else\myboxwidth=\halfwidth\fi
\vbox{%
\ifdim\myboxwidth=\hsize
\setbox\onelinebox=\hbox{%
\vbox{\hbox{%
$\Pi_{18,61}$ spans $L_{16.13}$%
}\hbox{%
$363222322223636322$%
}%
}%
\hfill\copy\matricesbox
}%
\ifdim\wd\onelinebox>\myboxwidth
\hbox to \myboxwidth{%
$\Pi_{18,61}$ spans $L_{16.13}$%
\hfil
$363222322223636322$%
}%
\box\matricesbox
\else
\hbox to \myboxwidth{%
\unhbox\onelinebox
}%
\fi
\else
\hbox to \myboxwidth{%
$\Pi_{18,61}$ spans $L_{16.13}$%
\hfil}%
\hbox to \myboxwidth{%
$363222322223636322$%
\hfil}%
\box\matricesbox
\fi
}%
\hfill\discretionary{}{}{}%
\setbox\matricesbox=\hbox{%
{$\left[\!\llap{\phantom{%
\begingroup \smaller\smaller\smaller
\endgroup%
}}\!\right]$}%
}%
\ifdim\wd\matricesbox>\halfwidth\myboxwidth=\hsize\else\myboxwidth=\halfwidth\fi
\vbox{%
\ifdim\myboxwidth=\hsize
\setbox\onelinebox=\hbox{%
\vbox{\hbox{%
$\Pi_{18,62}$ spans $L_{16.13}$%
}\hbox{%
$363222322232223636$%
}%
}%
\hfill\copy\matricesbox
}%
\ifdim\wd\onelinebox>\myboxwidth
\hbox to \myboxwidth{%
$\Pi_{18,62}$ spans $L_{16.13}$%
\hfil
$363222322232223636$%
}%
\box\matricesbox
\else
\hbox to \myboxwidth{%
\unhbox\onelinebox
}%
\fi
\else
\hbox to \myboxwidth{%
$\Pi_{18,62}$ spans $L_{16.13}$%
\hfil}%
\hbox to \myboxwidth{%
$363222322232223636$%
\hfil}%
\box\matricesbox
\fi
}%
\hfill\discretionary{}{}{}%
\setbox\matricesbox=\hbox{%
{$\left[\!\llap{\phantom{%
\begingroup \smaller\smaller\smaller
\endgroup%
}}\!\right]$}%
}%
\ifdim\wd\matricesbox>\halfwidth\myboxwidth=\hsize\else\myboxwidth=\halfwidth\fi
\vbox{%
\ifdim\myboxwidth=\hsize
\setbox\onelinebox=\hbox{%
\vbox{\hbox{%
$\Pi_{18,63}$ spans $L_{16.13}$%
}\hbox{%
$363222322232232236$%
}%
}%
\hfill\copy\matricesbox
}%
\ifdim\wd\onelinebox>\myboxwidth
\hbox to \myboxwidth{%
$\Pi_{18,63}$ spans $L_{16.13}$%
\hfil
$363222322232232236$%
}%
\box\matricesbox
\else
\hbox to \myboxwidth{%
\unhbox\onelinebox
}%
\fi
\else
\hbox to \myboxwidth{%
$\Pi_{18,63}$ spans $L_{16.13}$%
\hfil}%
\hbox to \myboxwidth{%
$363222322232232236$%
\hfil}%
\box\matricesbox
\fi
}%
\hfill\discretionary{}{}{}%
\setbox\matricesbox=\hbox{%
{$\left[\!\llap{\phantom{%
\begingroup \smaller\smaller\smaller
\endgroup%
}}\!\right]$}%
}%
\ifdim\wd\matricesbox>\halfwidth\myboxwidth=\hsize\else\myboxwidth=\halfwidth\fi
\vbox{%
\ifdim\myboxwidth=\hsize
\setbox\onelinebox=\hbox{%
\vbox{\hbox{%
$\Pi_{18,64}$ spans $L_{16.13}$%
}\hbox{%
$363222322232236322$%
}%
}%
\hfill\copy\matricesbox
}%
\ifdim\wd\onelinebox>\myboxwidth
\hbox to \myboxwidth{%
$\Pi_{18,64}$ spans $L_{16.13}$%
\hfil
$363222322232236322$%
}%
\box\matricesbox
\else
\hbox to \myboxwidth{%
\unhbox\onelinebox
}%
\fi
\else
\hbox to \myboxwidth{%
$\Pi_{18,64}$ spans $L_{16.13}$%
\hfil}%
\hbox to \myboxwidth{%
$363222322232236322$%
\hfil}%
\box\matricesbox
\fi
}%
\hfill\discretionary{}{}{}%
\setbox\matricesbox=\hbox{%
{$\left[\!\llap{\phantom{%
\begingroup \smaller\smaller\smaller
\endgroup%
}}\!\right]$}%
}%
\ifdim\wd\matricesbox>\halfwidth\myboxwidth=\hsize\else\myboxwidth=\halfwidth\fi
\vbox{%
\ifdim\myboxwidth=\hsize
\setbox\onelinebox=\hbox{%
\vbox{\hbox{%
$\Pi_{18,65}$ spans $L_{16.13}$%
}\hbox{%
$363222322236322236$%
}%
}%
\hfill\copy\matricesbox
}%
\ifdim\wd\onelinebox>\myboxwidth
\hbox to \myboxwidth{%
$\Pi_{18,65}$ spans $L_{16.13}$%
\hfil
$363222322236322236$%
}%
\box\matricesbox
\else
\hbox to \myboxwidth{%
\unhbox\onelinebox
}%
\fi
\else
\hbox to \myboxwidth{%
$\Pi_{18,65}$ spans $L_{16.13}$%
\hfil}%
\hbox to \myboxwidth{%
$363222322236322236$%
\hfil}%
\box\matricesbox
\fi
}%
\hfill\discretionary{}{}{}%
\setbox\matricesbox=\hbox{%
{$\left[\!\llap{\phantom{%
\begingroup \smaller\smaller\smaller
\endgroup%
}}\!\right]$}%
}%
\ifdim\wd\matricesbox>\halfwidth\myboxwidth=\hsize\else\myboxwidth=\halfwidth\fi
\vbox{%
\ifdim\myboxwidth=\hsize
\setbox\onelinebox=\hbox{%
\vbox{\hbox{%
$\Pi_{18,66}$ spans $L_{16.13}$%
}\hbox{%
$363222322322223636$%
}%
}%
\hfill\copy\matricesbox
}%
\ifdim\wd\onelinebox>\myboxwidth
\hbox to \myboxwidth{%
$\Pi_{18,66}$ spans $L_{16.13}$%
\hfil
$363222322322223636$%
}%
\box\matricesbox
\else
\hbox to \myboxwidth{%
\unhbox\onelinebox
}%
\fi
\else
\hbox to \myboxwidth{%
$\Pi_{18,66}$ spans $L_{16.13}$%
\hfil}%
\hbox to \myboxwidth{%
$363222322322223636$%
\hfil}%
\box\matricesbox
\fi
}%
\hfill\discretionary{}{}{}%
\setbox\matricesbox=\hbox{%
{$\left[\!\llap{\phantom{%
\begingroup \smaller\smaller\smaller
\endgroup%
}}\!\right]$}%
}%
\ifdim\wd\matricesbox>\halfwidth\myboxwidth=\hsize\else\myboxwidth=\halfwidth\fi
\vbox{%
\ifdim\myboxwidth=\hsize
\setbox\onelinebox=\hbox{%
\vbox{\hbox{%
$\Pi_{18,67}$ spans $L_{16.13}$%
}\hbox{%
$363222322322232236$%
}%
}%
\hfill\copy\matricesbox
}%
\ifdim\wd\onelinebox>\myboxwidth
\hbox to \myboxwidth{%
$\Pi_{18,67}$ spans $L_{16.13}$%
\hfil
$363222322322232236$%
}%
\box\matricesbox
\else
\hbox to \myboxwidth{%
\unhbox\onelinebox
}%
\fi
\else
\hbox to \myboxwidth{%
$\Pi_{18,67}$ spans $L_{16.13}$%
\hfil}%
\hbox to \myboxwidth{%
$363222322322232236$%
\hfil}%
\box\matricesbox
\fi
}%
\hfill\discretionary{}{}{}%
\setbox\matricesbox=\hbox{%
{$\left[\!\llap{\phantom{%
\begingroup \smaller\smaller\smaller
\endgroup%
}}\!\right]$}%
}%
\ifdim\wd\matricesbox>\halfwidth\myboxwidth=\hsize\else\myboxwidth=\halfwidth\fi
\vbox{%
\ifdim\myboxwidth=\hsize
\setbox\onelinebox=\hbox{%
\vbox{\hbox{%
$\Pi_{18,68}$ spans $L_{16.13}$%
}\hbox{%
$363222322322322322$%
}%
}%
\hfill\copy\matricesbox
}%
\ifdim\wd\onelinebox>\myboxwidth
\hbox to \myboxwidth{%
$\Pi_{18,68}$ spans $L_{16.13}$%
\hfil
$363222322322322322$%
}%
\box\matricesbox
\else
\hbox to \myboxwidth{%
\unhbox\onelinebox
}%
\fi
\else
\hbox to \myboxwidth{%
$\Pi_{18,68}$ spans $L_{16.13}$%
\hfil}%
\hbox to \myboxwidth{%
$363222322322322322$%
\hfil}%
\box\matricesbox
\fi
}%
\hfill\discretionary{}{}{}%
\setbox\matricesbox=\hbox{%
{$\left[\!\llap{\phantom{%
\begingroup \smaller\smaller\smaller
\endgroup%
}}\!\right]$}%
}%
\ifdim\wd\matricesbox>\halfwidth\myboxwidth=\hsize\else\myboxwidth=\halfwidth\fi
\vbox{%
\ifdim\myboxwidth=\hsize
\setbox\onelinebox=\hbox{%
\vbox{\hbox{%
$\Pi_{18,69}$ spans $L_{16.13}$%
}\hbox{%
$363222322322363222$%
}%
}%
\hfill\copy\matricesbox
}%
\ifdim\wd\onelinebox>\myboxwidth
\hbox to \myboxwidth{%
$\Pi_{18,69}$ spans $L_{16.13}$%
\hfil
$363222322322363222$%
}%
\box\matricesbox
\else
\hbox to \myboxwidth{%
\unhbox\onelinebox
}%
\fi
\else
\hbox to \myboxwidth{%
$\Pi_{18,69}$ spans $L_{16.13}$%
\hfil}%
\hbox to \myboxwidth{%
$363222322322363222$%
\hfil}%
\box\matricesbox
\fi
}%
\hfill\discretionary{}{}{}%
\setbox\matricesbox=\hbox{%
{$\left[\!\llap{\phantom{%
\begingroup \smaller\smaller\smaller
\endgroup%
}}\!\right]$}%
}%
\ifdim\wd\matricesbox>\halfwidth\myboxwidth=\hsize\else\myboxwidth=\halfwidth\fi
\vbox{%
\ifdim\myboxwidth=\hsize
\setbox\onelinebox=\hbox{%
\vbox{\hbox{%
$\Pi_{18,70}$ spans $L_{16.13}$%
}\hbox{%
$363222322363223222$%
}%
}%
\hfill\copy\matricesbox
}%
\ifdim\wd\onelinebox>\myboxwidth
\hbox to \myboxwidth{%
$\Pi_{18,70}$ spans $L_{16.13}$%
\hfil
$363222322363223222$%
}%
\box\matricesbox
\else
\hbox to \myboxwidth{%
\unhbox\onelinebox
}%
\fi
\else
\hbox to \myboxwidth{%
$\Pi_{18,70}$ spans $L_{16.13}$%
\hfil}%
\hbox to \myboxwidth{%
$363222322363223222$%
\hfil}%
\box\matricesbox
\fi
}%
\hfill\discretionary{}{}{}%
\setbox\matricesbox=\hbox{%
{$\left[\!\llap{\phantom{%
\begingroup \smaller\smaller\smaller
\endgroup%
}}\!\right]$}%
}%
\ifdim\wd\matricesbox>\halfwidth\myboxwidth=\hsize\else\myboxwidth=\halfwidth\fi
\vbox{%
\ifdim\myboxwidth=\hsize
\setbox\onelinebox=\hbox{%
\vbox{\hbox{%
$\Pi_{18,71}$ spans $L_{16.13}$%
}\hbox{%
$363222322363632222$%
}%
}%
\hfill\copy\matricesbox
}%
\ifdim\wd\onelinebox>\myboxwidth
\hbox to \myboxwidth{%
$\Pi_{18,71}$ spans $L_{16.13}$%
\hfil
$363222322363632222$%
}%
\box\matricesbox
\else
\hbox to \myboxwidth{%
\unhbox\onelinebox
}%
\fi
\else
\hbox to \myboxwidth{%
$\Pi_{18,71}$ spans $L_{16.13}$%
\hfil}%
\hbox to \myboxwidth{%
$363222322363632222$%
\hfil}%
\box\matricesbox
\fi
}%
\hfill\discretionary{}{}{}%
\setbox\matricesbox=\hbox{%
{$\left[\!\llap{\phantom{%
\begingroup \smaller\smaller\smaller
\endgroup%
}}\!\right]$}%
}%
\ifdim\wd\matricesbox>\halfwidth\myboxwidth=\hsize\else\myboxwidth=\halfwidth\fi
\vbox{%
\ifdim\myboxwidth=\hsize
\setbox\onelinebox=\hbox{%
\vbox{\hbox{%
$\Pi_{18,72}$ spans $L_{16.13}$%
}\hbox{%
$363222363223222322$%
}%
}%
\hfill\copy\matricesbox
}%
\ifdim\wd\onelinebox>\myboxwidth
\hbox to \myboxwidth{%
$\Pi_{18,72}$ spans $L_{16.13}$%
\hfil
$363222363223222322$%
}%
\box\matricesbox
\else
\hbox to \myboxwidth{%
\unhbox\onelinebox
}%
\fi
\else
\hbox to \myboxwidth{%
$\Pi_{18,72}$ spans $L_{16.13}$%
\hfil}%
\hbox to \myboxwidth{%
$363222363223222322$%
\hfil}%
\box\matricesbox
\fi
}%
\hfill\discretionary{}{}{}%
\setbox\matricesbox=\hbox{%
{$\left[\!\llap{\phantom{%
\begingroup \smaller\smaller\smaller
\endgroup%
}}\!\right]$}%
}%
\ifdim\wd\matricesbox>\halfwidth\myboxwidth=\hsize\else\myboxwidth=\halfwidth\fi
\vbox{%
\ifdim\myboxwidth=\hsize
\setbox\onelinebox=\hbox{%
\vbox{\hbox{%
$\Pi_{18,73}$ spans $L_{16.13}$%
}\hbox{%
$363222363223223222$%
}%
}%
\hfill\copy\matricesbox
}%
\ifdim\wd\onelinebox>\myboxwidth
\hbox to \myboxwidth{%
$\Pi_{18,73}$ spans $L_{16.13}$%
\hfil
$363222363223223222$%
}%
\box\matricesbox
\else
\hbox to \myboxwidth{%
\unhbox\onelinebox
}%
\fi
\else
\hbox to \myboxwidth{%
$\Pi_{18,73}$ spans $L_{16.13}$%
\hfil}%
\hbox to \myboxwidth{%
$363222363223223222$%
\hfil}%
\box\matricesbox
\fi
}%
\hfill\discretionary{}{}{}%
\setbox\matricesbox=\hbox{%
{$\left[\!\llap{\phantom{%
\begingroup \smaller\smaller\smaller
\endgroup%
}}\!\right]$}%
}%
\ifdim\wd\matricesbox>\halfwidth\myboxwidth=\hsize\else\myboxwidth=\halfwidth\fi
\vbox{%
\ifdim\myboxwidth=\hsize
\setbox\onelinebox=\hbox{%
\vbox{\hbox{%
$\Pi_{18,74}$ spans $L_{16.13}$%
}\hbox{%
$363222363223632222$%
}%
}%
\hfill\copy\matricesbox
}%
\ifdim\wd\onelinebox>\myboxwidth
\hbox to \myboxwidth{%
$\Pi_{18,74}$ spans $L_{16.13}$%
\hfil
$363222363223632222$%
}%
\box\matricesbox
\else
\hbox to \myboxwidth{%
\unhbox\onelinebox
}%
\fi
\else
\hbox to \myboxwidth{%
$\Pi_{18,74}$ spans $L_{16.13}$%
\hfil}%
\hbox to \myboxwidth{%
$363222363223632222$%
\hfil}%
\box\matricesbox
\fi
}%
\hfill\discretionary{}{}{}%
\setbox\matricesbox=\hbox{%
{$\left[\!\llap{\phantom{%
\begingroup \smaller\smaller\smaller
\endgroup%
}}\!\right]$}%
}%
\ifdim\wd\matricesbox>\halfwidth\myboxwidth=\hsize\else\myboxwidth=\halfwidth\fi
\vbox{%
\ifdim\myboxwidth=\hsize
\setbox\onelinebox=\hbox{%
\vbox{\hbox{%
$\Pi_{18,75}$ spans $L_{16.13}$%
}\hbox{%
$363222363632222322$%
}%
}%
\hfill\copy\matricesbox
}%
\ifdim\wd\onelinebox>\myboxwidth
\hbox to \myboxwidth{%
$\Pi_{18,75}$ spans $L_{16.13}$%
\hfil
$363222363632222322$%
}%
\box\matricesbox
\else
\hbox to \myboxwidth{%
\unhbox\onelinebox
}%
\fi
\else
\hbox to \myboxwidth{%
$\Pi_{18,75}$ spans $L_{16.13}$%
\hfil}%
\hbox to \myboxwidth{%
$363222363632222322$%
\hfil}%
\box\matricesbox
\fi
}%
\hfill\discretionary{}{}{}%
\setbox\matricesbox=\hbox{%
{$\left[\!\llap{\phantom{%
\begingroup \smaller\smaller\smaller
\endgroup%
}}\!\right]$}%
}%
\ifdim\wd\matricesbox>\halfwidth\myboxwidth=\hsize\else\myboxwidth=\halfwidth\fi
\vbox{%
\ifdim\myboxwidth=\hsize
\setbox\onelinebox=\hbox{%
\vbox{\hbox{%
$\Pi_{18,76}$ spans $L_{16.13}$%
}\hbox{%
$363222363632223222$%
}%
}%
\hfill\copy\matricesbox
}%
\ifdim\wd\onelinebox>\myboxwidth
\hbox to \myboxwidth{%
$\Pi_{18,76}$ spans $L_{16.13}$%
\hfil
$363222363632223222$%
}%
\box\matricesbox
\else
\hbox to \myboxwidth{%
\unhbox\onelinebox
}%
\fi
\else
\hbox to \myboxwidth{%
$\Pi_{18,76}$ spans $L_{16.13}$%
\hfil}%
\hbox to \myboxwidth{%
$363222363632223222$%
\hfil}%
\box\matricesbox
\fi
}%
\hfill\discretionary{}{}{}%
\setbox\matricesbox=\hbox{%
{$\left[\!\llap{\phantom{%
\begingroup \smaller\smaller\smaller
\endgroup%
}}\!\right]$}%
}%
\ifdim\wd\matricesbox>\halfwidth\myboxwidth=\hsize\else\myboxwidth=\halfwidth\fi
\vbox{%
\ifdim\myboxwidth=\hsize
\setbox\onelinebox=\hbox{%
\vbox{\hbox{%
$\Pi_{18,77}$ spans $L_{16.13}$%
}\hbox{%
$363222363632232222$%
}%
}%
\hfill\copy\matricesbox
}%
\ifdim\wd\onelinebox>\myboxwidth
\hbox to \myboxwidth{%
$\Pi_{18,77}$ spans $L_{16.13}$%
\hfil
$363222363632232222$%
}%
\box\matricesbox
\else
\hbox to \myboxwidth{%
\unhbox\onelinebox
}%
\fi
\else
\hbox to \myboxwidth{%
$\Pi_{18,77}$ spans $L_{16.13}$%
\hfil}%
\hbox to \myboxwidth{%
$363222363632232222$%
\hfil}%
\box\matricesbox
\fi
}%
\hfill\discretionary{}{}{}%
\setbox\matricesbox=\hbox{%
{$\left[\!\llap{\phantom{%
\begingroup \smaller\smaller\smaller
\endgroup%
}}\!\right]$}%
}%
\ifdim\wd\matricesbox>\halfwidth\myboxwidth=\hsize\else\myboxwidth=\halfwidth\fi
\vbox{%
\ifdim\myboxwidth=\hsize
\setbox\onelinebox=\hbox{%
\vbox{\hbox{%
$\Pi_{18,78}$ spans $L_{16.13}$%
}\hbox{%
$363223222223223636$%
}%
}%
\hfill\copy\matricesbox
}%
\ifdim\wd\onelinebox>\myboxwidth
\hbox to \myboxwidth{%
$\Pi_{18,78}$ spans $L_{16.13}$%
\hfil
$363223222223223636$%
}%
\box\matricesbox
\else
\hbox to \myboxwidth{%
\unhbox\onelinebox
}%
\fi
\else
\hbox to \myboxwidth{%
$\Pi_{18,78}$ spans $L_{16.13}$%
\hfil}%
\hbox to \myboxwidth{%
$363223222223223636$%
\hfil}%
\box\matricesbox
\fi
}%
\hfill\discretionary{}{}{}%
\setbox\matricesbox=\hbox{%
{$\left[\!\llap{\phantom{%
\begingroup \smaller\smaller\smaller
\endgroup%
}}\!\right]$}%
}%
\ifdim\wd\matricesbox>\halfwidth\myboxwidth=\hsize\else\myboxwidth=\halfwidth\fi
\vbox{%
\ifdim\myboxwidth=\hsize
\setbox\onelinebox=\hbox{%
\vbox{\hbox{%
$\Pi_{18,79}$ spans $L_{16.13}$%
}\hbox{%
$363223222223632236$%
}%
}%
\hfill\copy\matricesbox
}%
\ifdim\wd\onelinebox>\myboxwidth
\hbox to \myboxwidth{%
$\Pi_{18,79}$ spans $L_{16.13}$%
\hfil
$363223222223632236$%
}%
\box\matricesbox
\else
\hbox to \myboxwidth{%
\unhbox\onelinebox
}%
\fi
\else
\hbox to \myboxwidth{%
$\Pi_{18,79}$ spans $L_{16.13}$%
\hfil}%
\hbox to \myboxwidth{%
$363223222223632236$%
\hfil}%
\box\matricesbox
\fi
}%
\hfill\discretionary{}{}{}%
\setbox\matricesbox=\hbox{%
{$\left[\!\llap{\phantom{%
\begingroup \smaller\smaller\smaller
\endgroup%
}}\!\right]$}%
}%
\ifdim\wd\matricesbox>\halfwidth\myboxwidth=\hsize\else\myboxwidth=\halfwidth\fi
\vbox{%
\ifdim\myboxwidth=\hsize
\setbox\onelinebox=\hbox{%
\vbox{\hbox{%
$\Pi_{18,80}$ spans $L_{16.13}$%
}\hbox{%
$363223222223636322$%
}%
}%
\hfill\copy\matricesbox
}%
\ifdim\wd\onelinebox>\myboxwidth
\hbox to \myboxwidth{%
$\Pi_{18,80}$ spans $L_{16.13}$%
\hfil
$363223222223636322$%
}%
\box\matricesbox
\else
\hbox to \myboxwidth{%
\unhbox\onelinebox
}%
\fi
\else
\hbox to \myboxwidth{%
$\Pi_{18,80}$ spans $L_{16.13}$%
\hfil}%
\hbox to \myboxwidth{%
$363223222223636322$%
\hfil}%
\box\matricesbox
\fi
}%
\hfill\discretionary{}{}{}%
\setbox\matricesbox=\hbox{%
{$\left[\!\llap{\phantom{%
\begingroup \smaller\smaller\smaller
\endgroup%
}}\!\right]$}%
}%
\ifdim\wd\matricesbox>\halfwidth\myboxwidth=\hsize\else\myboxwidth=\halfwidth\fi
\vbox{%
\ifdim\myboxwidth=\hsize
\setbox\onelinebox=\hbox{%
\vbox{\hbox{%
$\Pi_{18,81}$ spans $L_{16.13}$%
}\hbox{%
$363223222232223636$%
}%
}%
\hfill\copy\matricesbox
}%
\ifdim\wd\onelinebox>\myboxwidth
\hbox to \myboxwidth{%
$\Pi_{18,81}$ spans $L_{16.13}$%
\hfil
$363223222232223636$%
}%
\box\matricesbox
\else
\hbox to \myboxwidth{%
\unhbox\onelinebox
}%
\fi
\else
\hbox to \myboxwidth{%
$\Pi_{18,81}$ spans $L_{16.13}$%
\hfil}%
\hbox to \myboxwidth{%
$363223222232223636$%
\hfil}%
\box\matricesbox
\fi
}%
\hfill\discretionary{}{}{}%
\setbox\matricesbox=\hbox{%
{$\left[\!\llap{\phantom{%
\begingroup \smaller\smaller\smaller
\endgroup%
}}\!\right]$}%
}%
\ifdim\wd\matricesbox>\halfwidth\myboxwidth=\hsize\else\myboxwidth=\halfwidth\fi
\vbox{%
\ifdim\myboxwidth=\hsize
\setbox\onelinebox=\hbox{%
\vbox{\hbox{%
$\Pi_{18,82}$ spans $L_{16.13}$%
}\hbox{%
$363223222232232236$%
}%
}%
\hfill\copy\matricesbox
}%
\ifdim\wd\onelinebox>\myboxwidth
\hbox to \myboxwidth{%
$\Pi_{18,82}$ spans $L_{16.13}$%
\hfil
$363223222232232236$%
}%
\box\matricesbox
\else
\hbox to \myboxwidth{%
\unhbox\onelinebox
}%
\fi
\else
\hbox to \myboxwidth{%
$\Pi_{18,82}$ spans $L_{16.13}$%
\hfil}%
\hbox to \myboxwidth{%
$363223222232232236$%
\hfil}%
\box\matricesbox
\fi
}%
\hfill\discretionary{}{}{}%
\setbox\matricesbox=\hbox{%
{$\left[\!\llap{\phantom{%
\begingroup \smaller\smaller\smaller
\endgroup%
}}\!\right]$}%
}%
\ifdim\wd\matricesbox>\halfwidth\myboxwidth=\hsize\else\myboxwidth=\halfwidth\fi
\vbox{%
\ifdim\myboxwidth=\hsize
\setbox\onelinebox=\hbox{%
\vbox{\hbox{%
$\Pi_{18,83}$ spans $L_{16.13}$%
}\hbox{%
$363223222322223636$%
}%
}%
\hfill\copy\matricesbox
}%
\ifdim\wd\onelinebox>\myboxwidth
\hbox to \myboxwidth{%
$\Pi_{18,83}$ spans $L_{16.13}$%
\hfil
$363223222322223636$%
}%
\box\matricesbox
\else
\hbox to \myboxwidth{%
\unhbox\onelinebox
}%
\fi
\else
\hbox to \myboxwidth{%
$\Pi_{18,83}$ spans $L_{16.13}$%
\hfil}%
\hbox to \myboxwidth{%
$363223222322223636$%
\hfil}%
\box\matricesbox
\fi
}%
\hfill\discretionary{}{}{}%
\setbox\matricesbox=\hbox{%
{$\left[\!\llap{\phantom{%
\begingroup \smaller\smaller\smaller
\endgroup%
}}\!\right]$}%
}%
\ifdim\wd\matricesbox>\halfwidth\myboxwidth=\hsize\else\myboxwidth=\halfwidth\fi
\vbox{%
\ifdim\myboxwidth=\hsize
\setbox\onelinebox=\hbox{%
\vbox{\hbox{%
$\Pi_{18,84}$ spans $L_{16.13}$%
}\hbox{%
$363223222322322322$%
}%
}%
\hfill\copy\matricesbox
}%
\ifdim\wd\onelinebox>\myboxwidth
\hbox to \myboxwidth{%
$\Pi_{18,84}$ spans $L_{16.13}$%
\hfil
$363223222322322322$%
}%
\box\matricesbox
\else
\hbox to \myboxwidth{%
\unhbox\onelinebox
}%
\fi
\else
\hbox to \myboxwidth{%
$\Pi_{18,84}$ spans $L_{16.13}$%
\hfil}%
\hbox to \myboxwidth{%
$363223222322322322$%
\hfil}%
\box\matricesbox
\fi
}%
\hfill\discretionary{}{}{}%
\setbox\matricesbox=\hbox{%
{$\left[\!\llap{\phantom{%
\begingroup \smaller\smaller\smaller
\endgroup%
}}\!\right]$}%
}%
\ifdim\wd\matricesbox>\halfwidth\myboxwidth=\hsize\else\myboxwidth=\halfwidth\fi
\vbox{%
\ifdim\myboxwidth=\hsize
\setbox\onelinebox=\hbox{%
\vbox{\hbox{%
$\Pi_{18,85}$ spans $L_{16.13}$%
}\hbox{%
$363223222363632222$%
}%
}%
\hfill\copy\matricesbox
}%
\ifdim\wd\onelinebox>\myboxwidth
\hbox to \myboxwidth{%
$\Pi_{18,85}$ spans $L_{16.13}$%
\hfil
$363223222363632222$%
}%
\box\matricesbox
\else
\hbox to \myboxwidth{%
\unhbox\onelinebox
}%
\fi
\else
\hbox to \myboxwidth{%
$\Pi_{18,85}$ spans $L_{16.13}$%
\hfil}%
\hbox to \myboxwidth{%
$363223222363632222$%
\hfil}%
\box\matricesbox
\fi
}%
\hfill\discretionary{}{}{}%
\setbox\matricesbox=\hbox{%
{$\left[\!\llap{\phantom{%
\begingroup \smaller\smaller\smaller
\endgroup%
}}\!\right]$}%
}%
\ifdim\wd\matricesbox>\halfwidth\myboxwidth=\hsize\else\myboxwidth=\halfwidth\fi
\vbox{%
\ifdim\myboxwidth=\hsize
\setbox\onelinebox=\hbox{%
\vbox{\hbox{%
$\Pi_{18,86}$ spans $L_{16.13}$%
}\hbox{%
$363223223222223636$%
}%
}%
\hfill\copy\matricesbox
}%
\ifdim\wd\onelinebox>\myboxwidth
\hbox to \myboxwidth{%
$\Pi_{18,86}$ spans $L_{16.13}$%
\hfil
$363223223222223636$%
}%
\box\matricesbox
\else
\hbox to \myboxwidth{%
\unhbox\onelinebox
}%
\fi
\else
\hbox to \myboxwidth{%
$\Pi_{18,86}$ spans $L_{16.13}$%
\hfil}%
\hbox to \myboxwidth{%
$363223223222223636$%
\hfil}%
\box\matricesbox
\fi
}%
\hfill\discretionary{}{}{}%
\setbox\matricesbox=\hbox{%
{$\left[\!\llap{\phantom{%
\begingroup \smaller\smaller\smaller
\endgroup%
}}\!\right]$}%
}%
\ifdim\wd\matricesbox>\halfwidth\myboxwidth=\hsize\else\myboxwidth=\halfwidth\fi
\vbox{%
\ifdim\myboxwidth=\hsize
\setbox\onelinebox=\hbox{%
\vbox{\hbox{%
$\Pi_{18,87}$ spans $L_{16.13}$%
}\hbox{%
$363223223222322322$%
}%
}%
\hfill\copy\matricesbox
}%
\ifdim\wd\onelinebox>\myboxwidth
\hbox to \myboxwidth{%
$\Pi_{18,87}$ spans $L_{16.13}$%
\hfil
$363223223222322322$%
}%
\box\matricesbox
\else
\hbox to \myboxwidth{%
\unhbox\onelinebox
}%
\fi
\else
\hbox to \myboxwidth{%
$\Pi_{18,87}$ spans $L_{16.13}$%
\hfil}%
\hbox to \myboxwidth{%
$363223223222322322$%
\hfil}%
\box\matricesbox
\fi
}%
\hfill\discretionary{}{}{}%
\setbox\matricesbox=\hbox{%
{$\left[\!\llap{\phantom{%
\begingroup \smaller\smaller\smaller
\endgroup%
}}\!\right]$}%
}%
\ifdim\wd\matricesbox>\halfwidth\myboxwidth=\hsize\else\myboxwidth=\halfwidth\fi
\vbox{%
\ifdim\myboxwidth=\hsize
\setbox\onelinebox=\hbox{%
\vbox{\hbox{%
$\Pi_{18,88}$ spans $L_{16.13}$%
}\hbox{%
$363223223223222322$%
}%
}%
\hfill\copy\matricesbox
}%
\ifdim\wd\onelinebox>\myboxwidth
\hbox to \myboxwidth{%
$\Pi_{18,88}$ spans $L_{16.13}$%
\hfil
$363223223223222322$%
}%
\box\matricesbox
\else
\hbox to \myboxwidth{%
\unhbox\onelinebox
}%
\fi
\else
\hbox to \myboxwidth{%
$\Pi_{18,88}$ spans $L_{16.13}$%
\hfil}%
\hbox to \myboxwidth{%
$363223223223222322$%
\hfil}%
\box\matricesbox
\fi
}%
\hfill\discretionary{}{}{}%
\setbox\matricesbox=\hbox{%
{$\left[\!\llap{\phantom{%
\begingroup \smaller\smaller\smaller
\endgroup%
}}\!\right]$}%
}%
\ifdim\wd\matricesbox>\halfwidth\myboxwidth=\hsize\else\myboxwidth=\halfwidth\fi
\vbox{%
\ifdim\myboxwidth=\hsize
\setbox\onelinebox=\hbox{%
\vbox{\hbox{%
$\Pi_{18,89}$ spans $L_{16.13}$%
}\hbox{%
$363223223223223222$%
}%
}%
\hfill\copy\matricesbox
}%
\ifdim\wd\onelinebox>\myboxwidth
\hbox to \myboxwidth{%
$\Pi_{18,89}$ spans $L_{16.13}$%
\hfil
$363223223223223222$%
}%
\box\matricesbox
\else
\hbox to \myboxwidth{%
\unhbox\onelinebox
}%
\fi
\else
\hbox to \myboxwidth{%
$\Pi_{18,89}$ spans $L_{16.13}$%
\hfil}%
\hbox to \myboxwidth{%
$363223223223223222$%
\hfil}%
\box\matricesbox
\fi
}%
\hfill\discretionary{}{}{}%
\setbox\matricesbox=\hbox{%
{$\left[\!\llap{\phantom{%
\begingroup \smaller\smaller\smaller
\endgroup%
}}\!\right]$}%
}%
\ifdim\wd\matricesbox>\halfwidth\myboxwidth=\hsize\else\myboxwidth=\halfwidth\fi
\vbox{%
\ifdim\myboxwidth=\hsize
\setbox\onelinebox=\hbox{%
\vbox{\hbox{%
$\Pi_{18,90}$ spans $L_{16.13}$%
}\hbox{%
$363223223632222322$%
}%
}%
\hfill\copy\matricesbox
}%
\ifdim\wd\onelinebox>\myboxwidth
\hbox to \myboxwidth{%
$\Pi_{18,90}$ spans $L_{16.13}$%
\hfil
$363223223632222322$%
}%
\box\matricesbox
\else
\hbox to \myboxwidth{%
\unhbox\onelinebox
}%
\fi
\else
\hbox to \myboxwidth{%
$\Pi_{18,90}$ spans $L_{16.13}$%
\hfil}%
\hbox to \myboxwidth{%
$363223223632222322$%
\hfil}%
\box\matricesbox
\fi
}%
\hfill\discretionary{}{}{}%
\setbox\matricesbox=\hbox{%
{$\left[\!\llap{\phantom{%
\begingroup \smaller\smaller\smaller
\endgroup%
}}\!\right]$}%
}%
\ifdim\wd\matricesbox>\halfwidth\myboxwidth=\hsize\else\myboxwidth=\halfwidth\fi
\vbox{%
\ifdim\myboxwidth=\hsize
\setbox\onelinebox=\hbox{%
\vbox{\hbox{%
$\Pi_{18,91}$ spans $L_{16.13}$%
}\hbox{%
$363223632222322322$%
}%
}%
\hfill\copy\matricesbox
}%
\ifdim\wd\onelinebox>\myboxwidth
\hbox to \myboxwidth{%
$\Pi_{18,91}$ spans $L_{16.13}$%
\hfil
$363223632222322322$%
}%
\box\matricesbox
\else
\hbox to \myboxwidth{%
\unhbox\onelinebox
}%
\fi
\else
\hbox to \myboxwidth{%
$\Pi_{18,91}$ spans $L_{16.13}$%
\hfil}%
\hbox to \myboxwidth{%
$363223632222322322$%
\hfil}%
\box\matricesbox
\fi
}%
\hfill\discretionary{}{}{}%
\setbox\matricesbox=\hbox{%
{$\left[\!\llap{\phantom{%
\begingroup \smaller\smaller\smaller
\endgroup%
}}\!\right]$}%
}%
\ifdim\wd\matricesbox>\halfwidth\myboxwidth=\hsize\else\myboxwidth=\halfwidth\fi
\vbox{%
\ifdim\myboxwidth=\hsize
\setbox\onelinebox=\hbox{%
\vbox{\hbox{%
$\Pi_{18,92}$ spans $L_{16.13}$%
}\hbox{%
$363223632223222322$%
}%
}%
\hfill\copy\matricesbox
}%
\ifdim\wd\onelinebox>\myboxwidth
\hbox to \myboxwidth{%
$\Pi_{18,92}$ spans $L_{16.13}$%
\hfil
$363223632223222322$%
}%
\box\matricesbox
\else
\hbox to \myboxwidth{%
\unhbox\onelinebox
}%
\fi
\else
\hbox to \myboxwidth{%
$\Pi_{18,92}$ spans $L_{16.13}$%
\hfil}%
\hbox to \myboxwidth{%
$363223632223222322$%
\hfil}%
\box\matricesbox
\fi
}%
\hfill\discretionary{}{}{}%
\setbox\matricesbox=\hbox{%
{$\left[\!\llap{\phantom{%
\begingroup \smaller\smaller\smaller
\endgroup%
}}\!\right]$}%
}%
\ifdim\wd\matricesbox>\halfwidth\myboxwidth=\hsize\else\myboxwidth=\halfwidth\fi
\vbox{%
\ifdim\myboxwidth=\hsize
\setbox\onelinebox=\hbox{%
\vbox{\hbox{%
$\Pi_{18,93}$ spans $L_{16.13}$%
}\hbox{%
$363632222232223636$%
}%
}%
\hfill\copy\matricesbox
}%
\ifdim\wd\onelinebox>\myboxwidth
\hbox to \myboxwidth{%
$\Pi_{18,93}$ spans $L_{16.13}$%
\hfil
$363632222232223636$%
}%
\box\matricesbox
\else
\hbox to \myboxwidth{%
\unhbox\onelinebox
}%
\fi
\else
\hbox to \myboxwidth{%
$\Pi_{18,93}$ spans $L_{16.13}$%
\hfil}%
\hbox to \myboxwidth{%
$363632222232223636$%
\hfil}%
\box\matricesbox
\fi
}%
\hfill\discretionary{}{}{}%
\setbox\matricesbox=\hbox{%
{$\left[\!\llap{\phantom{%
\begingroup \smaller\smaller\smaller
\endgroup%
}}\!\right]$}%
}%
\ifdim\wd\matricesbox>\halfwidth\myboxwidth=\hsize\else\myboxwidth=\halfwidth\fi
\vbox{%
\ifdim\myboxwidth=\hsize
\setbox\onelinebox=\hbox{%
\vbox{\hbox{%
$\Pi_{18,94}$ spans $L_{16.13}$%
}\hbox{%
$363632222322322322$%
}%
}%
\hfill\copy\matricesbox
}%
\ifdim\wd\onelinebox>\myboxwidth
\hbox to \myboxwidth{%
$\Pi_{18,94}$ spans $L_{16.13}$%
\hfil
$363632222322322322$%
}%
\box\matricesbox
\else
\hbox to \myboxwidth{%
\unhbox\onelinebox
}%
\fi
\else
\hbox to \myboxwidth{%
$\Pi_{18,94}$ spans $L_{16.13}$%
\hfil}%
\hbox to \myboxwidth{%
$363632222322322322$%
\hfil}%
\box\matricesbox
\fi
}%
\hfill\discretionary{}{}{}%
\setbox\matricesbox=\hbox{%
{$\left[\!\llap{\phantom{%
\begingroup \smaller\smaller\smaller
\endgroup%
}}\!\right]$}%
}%
\ifdim\wd\matricesbox>\halfwidth\myboxwidth=\hsize\else\myboxwidth=\halfwidth\fi
\vbox{%
\ifdim\myboxwidth=\hsize
\setbox\onelinebox=\hbox{%
\vbox{\hbox{%
$\Pi_{18,95}$ spans $L_{16.13}$%
}\hbox{%
$363632223222322322$%
}%
}%
\hfill\copy\matricesbox
}%
\ifdim\wd\onelinebox>\myboxwidth
\hbox to \myboxwidth{%
$\Pi_{18,95}$ spans $L_{16.13}$%
\hfil
$363632223222322322$%
}%
\box\matricesbox
\else
\hbox to \myboxwidth{%
\unhbox\onelinebox
}%
\fi
\else
\hbox to \myboxwidth{%
$\Pi_{18,95}$ spans $L_{16.13}$%
\hfil}%
\hbox to \myboxwidth{%
$363632223222322322$%
\hfil}%
\box\matricesbox
\fi
}%
\hfill\discretionary{}{}{}%
\setbox\matricesbox=\hbox{%
{$\left[\!\llap{\phantom{%
\begingroup \smaller\smaller\smaller
\endgroup%
}}\!\right]$}%
}%
\ifdim\wd\matricesbox>\halfwidth\myboxwidth=\hsize\else\myboxwidth=\halfwidth\fi
\vbox{%
\ifdim\myboxwidth=\hsize
\setbox\onelinebox=\hbox{%
\vbox{\hbox{%
$\Pi_{18,96}$ spans $L_{16.13}$%
}\hbox{%
$363632223223222322$%
}%
}%
\hfill\copy\matricesbox
}%
\ifdim\wd\onelinebox>\myboxwidth
\hbox to \myboxwidth{%
$\Pi_{18,96}$ spans $L_{16.13}$%
\hfil
$363632223223222322$%
}%
\box\matricesbox
\else
\hbox to \myboxwidth{%
\unhbox\onelinebox
}%
\fi
\else
\hbox to \myboxwidth{%
$\Pi_{18,96}$ spans $L_{16.13}$%
\hfil}%
\hbox to \myboxwidth{%
$363632223223222322$%
\hfil}%
\box\matricesbox
\fi
}%
\hfill\discretionary{}{}{}%
\setbox\matricesbox=\hbox{%
{$\left[\!\llap{\phantom{%
\begingroup \smaller\smaller\smaller
\endgroup%
}}\!\right]$}%
}%
\ifdim\wd\matricesbox>\halfwidth\myboxwidth=\hsize\else\myboxwidth=\halfwidth\fi
\vbox{%
\ifdim\myboxwidth=\hsize
\setbox\onelinebox=\hbox{%
\vbox{\hbox{%
$\Pi_{18,97}$ spans $L_{16.13}$%
}\hbox{%
$363632232222322322$%
}%
}%
\hfill\copy\matricesbox
}%
\ifdim\wd\onelinebox>\myboxwidth
\hbox to \myboxwidth{%
$\Pi_{18,97}$ spans $L_{16.13}$%
\hfil
$363632232222322322$%
}%
\box\matricesbox
\else
\hbox to \myboxwidth{%
\unhbox\onelinebox
}%
\fi
\else
\hbox to \myboxwidth{%
$\Pi_{18,97}$ spans $L_{16.13}$%
\hfil}%
\hbox to \myboxwidth{%
$363632232222322322$%
\hfil}%
\box\matricesbox
\fi
}%
\hfill\discretionary{}{}{}%
\setbox\matricesbox=\hbox{%
{$\left[\!\llap{\phantom{%
\begingroup \smaller\smaller\smaller
\endgroup%
}}\!\right]$}%
}%
\ifdim\wd\matricesbox>\halfwidth\myboxwidth=\hsize\else\myboxwidth=\halfwidth\fi
\vbox{%
\ifdim\myboxwidth=\hsize
\setbox\onelinebox=\hbox{%
\vbox{\hbox{%
$\Pi_{18,98}$ spans $L_{142.20}$%
}\hbox{%
$\infty422\infty4224\infty4224\infty422$%
}%
}%
\hfill\copy\matricesbox
}%
\ifdim\wd\onelinebox>\myboxwidth
\hbox to \myboxwidth{%
$\Pi_{18,98}$ spans $L_{142.20}$%
\hfil
$\infty422\infty4224\infty4224\infty422$%
}%
\box\matricesbox
\else
\hbox to \myboxwidth{%
\unhbox\onelinebox
}%
\fi
\else
\hbox to \myboxwidth{%
$\Pi_{18,98}$ spans $L_{142.20}$%
\hfil}%
\hbox to \myboxwidth{%
$\infty422\infty4224\infty4224\infty422$%
\hfil}%
\box\matricesbox
\fi
}%
\hfill\discretionary{}{}{}%

\vskip2pt\hrule\vskip2pt

\leavevmode\setbox\matricesbox=\hbox{%
{$\left[\!\llap{\phantom{%
\begingroup \smaller\smaller\smaller\begin{tabular}{@{}c@{}}%
\phantom{0}\\\phantom{0}\\\phantom{0}
\end{tabular}\endgroup%
}}\right.$}%
\begingroup \smaller\smaller\smaller\begin{tabular}{@{}c@{}}%
-1\\\phantom{0}\\\phantom{0}
\end{tabular}\endgroup%
\kern3pt%
\begingroup \smaller\smaller\smaller\begin{tabular}{@{}c@{}}%
\phantom{0}\\15/2\\\phantom{0}
\end{tabular}\endgroup%
\kern3pt%
\begingroup \smaller\smaller\smaller\begin{tabular}{@{}c@{}}%
\phantom{0}\\\phantom{0}\\45/2
\end{tabular}\endgroup%
{$\left.\llap{\phantom{%
\begingroup \smaller\smaller\smaller\begin{tabular}{@{}c@{}}%
\phantom{0}\\\phantom{0}\\\phantom{0}
\end{tabular}\endgroup%
}}\!\right]$}%
{$\left[\!\llap{\phantom{%
\begingroup \smaller\smaller\smaller\begin{tabular}{@{}c@{}}%
0\\0\\0
\end{tabular}\endgroup%
}}\right.$}%
\begingroup \smaller\smaller\smaller\begin{tabular}{@{}c@{}}%
5\\-2\\0
\end{tabular}\endgroup%
\kern3pt%
\begingroup \smaller\smaller\smaller\begin{tabular}{@{}c@{}}%
9\\-3\\1
\end{tabular}\endgroup%
\kern3pt%
\begingroup \smaller\smaller\smaller\begin{tabular}{@{}c@{}}%
5\\-1\\1
\end{tabular}\endgroup%
\kern3pt%
\begingroup \smaller\smaller\smaller\begin{tabular}{@{}c@{}}%
90\\-3\\19
\end{tabular}\endgroup%
\kern3pt%
\begingroup \smaller\smaller\smaller\begin{tabular}{@{}c@{}}%
90\\3\\19
\end{tabular}\endgroup%
\kern3pt%
\begingroup \smaller\smaller\smaller\begin{tabular}{@{}c@{}}%
30\\4\\6
\end{tabular}\endgroup%
\kern3pt%
\begingroup \smaller\smaller\smaller\begin{tabular}{@{}c@{}}%
30\\7\\5
\end{tabular}\endgroup%
\kern3pt%
\begingroup \smaller\smaller\smaller\begin{tabular}{@{}c@{}}%
90\\27\\11
\end{tabular}\endgroup%
\kern3pt%
\begingroup \smaller\smaller\smaller\begin{tabular}{@{}c@{}}%
90\\30\\8
\end{tabular}\endgroup%
\kern3pt%
\begingroup \smaller\smaller\smaller\begin{tabular}{@{}c@{}}%
30\\11\\1
\end{tabular}\endgroup%
{$\left.\llap{\phantom{%
\begingroup \smaller\smaller\smaller\begin{tabular}{@{}c@{}}%
0\\0\\0
\end{tabular}\endgroup%
}}\!\right]$}%
}%
\ifdim\wd\matricesbox>\halfwidth\myboxwidth=\hsize\else\myboxwidth=\halfwidth\fi
\vbox{%
\ifdim\myboxwidth=\hsize
\setbox\onelinebox=\hbox{%
\vbox{\hbox{%
$\Pi_{19,1}$ spans $L_{16.13}$%
}\hbox{%
$22|222363636\slashthree6363632\rtimes D_{2}$%
}%
}%
\hfill\copy\matricesbox
}%
\ifdim\wd\onelinebox>\myboxwidth
\hbox to \myboxwidth{%
$\Pi_{19,1}$ spans $L_{16.13}$%
\hfil
$22|222363636\slashthree6363632\rtimes D_{2}$%
}%
\box\matricesbox
\else
\hbox to \myboxwidth{%
\unhbox\onelinebox
}%
\fi
\else
\hbox to \myboxwidth{%
$\Pi_{19,1}$ spans $L_{16.13}$%
\hfil}%
\hbox to \myboxwidth{%
$22|222363636\slashthree6363632\rtimes D_{2}$%
\hfil}%
\box\matricesbox
\fi
}%
\hfill\discretionary{}{}{}%
\setbox\matricesbox=\hbox{%
{$\left[\!\llap{\phantom{%
\begingroup \smaller\smaller\smaller\begin{tabular}{@{}c@{}}%
\phantom{0}\\\phantom{0}\\\phantom{0}
\end{tabular}\endgroup%
}}\right.$}%
\begingroup \smaller\smaller\smaller\begin{tabular}{@{}c@{}}%
-1\\\phantom{0}\\\phantom{0}
\end{tabular}\endgroup%
\kern3pt%
\begingroup \smaller\smaller\smaller\begin{tabular}{@{}c@{}}%
\phantom{0}\\45/2\\\phantom{0}
\end{tabular}\endgroup%
\kern3pt%
\begingroup \smaller\smaller\smaller\begin{tabular}{@{}c@{}}%
\phantom{0}\\\phantom{0}\\15/2
\end{tabular}\endgroup%
{$\left.\llap{\phantom{%
\begingroup \smaller\smaller\smaller\begin{tabular}{@{}c@{}}%
\phantom{0}\\\phantom{0}\\\phantom{0}
\end{tabular}\endgroup%
}}\!\right]$}%
{$\left[\!\llap{\phantom{%
\begingroup \smaller\smaller\smaller\begin{tabular}{@{}c@{}}%
0\\0\\0
\end{tabular}\endgroup%
}}\right.$}%
\begingroup \smaller\smaller\smaller\begin{tabular}{@{}c@{}}%
9\\2\\0
\end{tabular}\endgroup%
\kern3pt%
\begingroup \smaller\smaller\smaller\begin{tabular}{@{}c@{}}%
5\\1\\1
\end{tabular}\endgroup%
\kern3pt%
\begingroup \smaller\smaller\smaller\begin{tabular}{@{}c@{}}%
9\\1\\3
\end{tabular}\endgroup%
\kern3pt%
\begingroup \smaller\smaller\smaller\begin{tabular}{@{}c@{}}%
30\\1\\11
\end{tabular}\endgroup%
\kern3pt%
\begingroup \smaller\smaller\smaller\begin{tabular}{@{}c@{}}%
30\\-1\\11
\end{tabular}\endgroup%
\kern3pt%
\begingroup \smaller\smaller\smaller\begin{tabular}{@{}c@{}}%
90\\-8\\30
\end{tabular}\endgroup%
\kern3pt%
\begingroup \smaller\smaller\smaller\begin{tabular}{@{}c@{}}%
90\\-11\\27
\end{tabular}\endgroup%
\kern3pt%
\begingroup \smaller\smaller\smaller\begin{tabular}{@{}c@{}}%
30\\-5\\7
\end{tabular}\endgroup%
\kern3pt%
\begingroup \smaller\smaller\smaller\begin{tabular}{@{}c@{}}%
30\\-6\\4
\end{tabular}\endgroup%
\kern3pt%
\begingroup \smaller\smaller\smaller\begin{tabular}{@{}c@{}}%
90\\-19\\3
\end{tabular}\endgroup%
{$\left.\llap{\phantom{%
\begingroup \smaller\smaller\smaller\begin{tabular}{@{}c@{}}%
0\\0\\0
\end{tabular}\endgroup%
}}\!\right]$}%
}%
\ifdim\wd\matricesbox>\halfwidth\myboxwidth=\hsize\else\myboxwidth=\halfwidth\fi
\vbox{%
\ifdim\myboxwidth=\hsize
\setbox\onelinebox=\hbox{%
\vbox{\hbox{%
$\Pi_{19,2}$ spans $L_{16.13}$%
}\hbox{%
$363222|222363636\slashthree636\rtimes D_{2}$%
}%
}%
\hfill\copy\matricesbox
}%
\ifdim\wd\onelinebox>\myboxwidth
\hbox to \myboxwidth{%
$\Pi_{19,2}$ spans $L_{16.13}$%
\hfil
$363222|222363636\slashthree636\rtimes D_{2}$%
}%
\box\matricesbox
\else
\hbox to \myboxwidth{%
\unhbox\onelinebox
}%
\fi
\else
\hbox to \myboxwidth{%
$\Pi_{19,2}$ spans $L_{16.13}$%
\hfil}%
\hbox to \myboxwidth{%
$363222|222363636\slashthree636\rtimes D_{2}$%
\hfil}%
\box\matricesbox
\fi
}%
\hfill\discretionary{}{}{}%
\setbox\matricesbox=\hbox{%
{$\left[\!\llap{\phantom{%
\begingroup \smaller\smaller\smaller\begin{tabular}{@{}c@{}}%
\phantom{0}\\\phantom{0}\\\phantom{0}
\end{tabular}\endgroup%
}}\right.$}%
\begingroup \smaller\smaller\smaller\begin{tabular}{@{}c@{}}%
-1\\\phantom{0}\\\phantom{0}
\end{tabular}\endgroup%
\kern3pt%
\begingroup \smaller\smaller\smaller\begin{tabular}{@{}c@{}}%
\phantom{0}\\15/2\\\phantom{0}
\end{tabular}\endgroup%
\kern3pt%
\begingroup \smaller\smaller\smaller\begin{tabular}{@{}c@{}}%
\phantom{0}\\\phantom{0}\\45/2
\end{tabular}\endgroup%
{$\left.\llap{\phantom{%
\begingroup \smaller\smaller\smaller\begin{tabular}{@{}c@{}}%
\phantom{0}\\\phantom{0}\\\phantom{0}
\end{tabular}\endgroup%
}}\!\right]$}%
{$\left[\!\llap{\phantom{%
\begingroup \smaller\smaller\smaller\begin{tabular}{@{}c@{}}%
0\\0\\0
\end{tabular}\endgroup%
}}\right.$}%
\begingroup \smaller\smaller\smaller\begin{tabular}{@{}c@{}}%
5\\-2\\0
\end{tabular}\endgroup%
\kern3pt%
\begingroup \smaller\smaller\smaller\begin{tabular}{@{}c@{}}%
9\\-3\\-1
\end{tabular}\endgroup%
\kern3pt%
\begingroup \smaller\smaller\smaller\begin{tabular}{@{}c@{}}%
30\\-7\\-5
\end{tabular}\endgroup%
\kern3pt%
\begingroup \smaller\smaller\smaller\begin{tabular}{@{}c@{}}%
30\\-4\\-6
\end{tabular}\endgroup%
\kern3pt%
\begingroup \smaller\smaller\smaller\begin{tabular}{@{}c@{}}%
90\\-3\\-19
\end{tabular}\endgroup%
\kern3pt%
\begingroup \smaller\smaller\smaller\begin{tabular}{@{}c@{}}%
90\\3\\-19
\end{tabular}\endgroup%
\kern3pt%
\begingroup \smaller\smaller\smaller\begin{tabular}{@{}c@{}}%
5\\1\\-1
\end{tabular}\endgroup%
\kern3pt%
\begingroup \smaller\smaller\smaller\begin{tabular}{@{}c@{}}%
90\\27\\-11
\end{tabular}\endgroup%
\kern3pt%
\begingroup \smaller\smaller\smaller\begin{tabular}{@{}c@{}}%
90\\30\\-8
\end{tabular}\endgroup%
\kern3pt%
\begingroup \smaller\smaller\smaller\begin{tabular}{@{}c@{}}%
30\\11\\-1
\end{tabular}\endgroup%
{$\left.\llap{\phantom{%
\begingroup \smaller\smaller\smaller\begin{tabular}{@{}c@{}}%
0\\0\\0
\end{tabular}\endgroup%
}}\!\right]$}%
}%
\ifdim\wd\matricesbox>\halfwidth\myboxwidth=\hsize\else\myboxwidth=\halfwidth\fi
\vbox{%
\ifdim\myboxwidth=\hsize
\setbox\onelinebox=\hbox{%
\vbox{\hbox{%
$\Pi_{19,3}$ spans $L_{16.13}$%
}\hbox{%
$36322|223632236\slashthree6322\rtimes D_{2}$%
}%
}%
\hfill\copy\matricesbox
}%
\ifdim\wd\onelinebox>\myboxwidth
\hbox to \myboxwidth{%
$\Pi_{19,3}$ spans $L_{16.13}$%
\hfil
$36322|223632236\slashthree6322\rtimes D_{2}$%
}%
\box\matricesbox
\else
\hbox to \myboxwidth{%
\unhbox\onelinebox
}%
\fi
\else
\hbox to \myboxwidth{%
$\Pi_{19,3}$ spans $L_{16.13}$%
\hfil}%
\hbox to \myboxwidth{%
$36322|223632236\slashthree6322\rtimes D_{2}$%
\hfil}%
\box\matricesbox
\fi
}%
\hfill\discretionary{}{}{}%
\setbox\matricesbox=\hbox{%
{$\left[\!\llap{\phantom{%
\begingroup \smaller\smaller\smaller\begin{tabular}{@{}c@{}}%
\phantom{0}\\\phantom{0}\\\phantom{0}
\end{tabular}\endgroup%
}}\right.$}%
\begingroup \smaller\smaller\smaller\begin{tabular}{@{}c@{}}%
-1\\\phantom{0}\\\phantom{0}
\end{tabular}\endgroup%
\kern3pt%
\begingroup \smaller\smaller\smaller\begin{tabular}{@{}c@{}}%
\phantom{0}\\15/2\\\phantom{0}
\end{tabular}\endgroup%
\kern3pt%
\begingroup \smaller\smaller\smaller\begin{tabular}{@{}c@{}}%
\phantom{0}\\\phantom{0}\\45/2
\end{tabular}\endgroup%
{$\left.\llap{\phantom{%
\begingroup \smaller\smaller\smaller\begin{tabular}{@{}c@{}}%
\phantom{0}\\\phantom{0}\\\phantom{0}
\end{tabular}\endgroup%
}}\!\right]$}%
{$\left[\!\llap{\phantom{%
\begingroup \smaller\smaller\smaller\begin{tabular}{@{}c@{}}%
0\\0\\0
\end{tabular}\endgroup%
}}\right.$}%
\begingroup \smaller\smaller\smaller\begin{tabular}{@{}c@{}}%
5\\2\\0
\end{tabular}\endgroup%
\kern3pt%
\begingroup \smaller\smaller\smaller\begin{tabular}{@{}c@{}}%
9\\3\\-1
\end{tabular}\endgroup%
\kern3pt%
\begingroup \smaller\smaller\smaller\begin{tabular}{@{}c@{}}%
30\\7\\-5
\end{tabular}\endgroup%
\kern3pt%
\begingroup \smaller\smaller\smaller\begin{tabular}{@{}c@{}}%
30\\4\\-6
\end{tabular}\endgroup%
\kern3pt%
\begingroup \smaller\smaller\smaller\begin{tabular}{@{}c@{}}%
90\\3\\-19
\end{tabular}\endgroup%
\kern3pt%
\begingroup \smaller\smaller\smaller\begin{tabular}{@{}c@{}}%
90\\-3\\-19
\end{tabular}\endgroup%
\kern3pt%
\begingroup \smaller\smaller\smaller\begin{tabular}{@{}c@{}}%
30\\-4\\-6
\end{tabular}\endgroup%
\kern3pt%
\begingroup \smaller\smaller\smaller\begin{tabular}{@{}c@{}}%
30\\-7\\-5
\end{tabular}\endgroup%
\kern3pt%
\begingroup \smaller\smaller\smaller\begin{tabular}{@{}c@{}}%
9\\-3\\-1
\end{tabular}\endgroup%
\kern3pt%
\begingroup \smaller\smaller\smaller\begin{tabular}{@{}c@{}}%
30\\-11\\-1
\end{tabular}\endgroup%
{$\left.\llap{\phantom{%
\begingroup \smaller\smaller\smaller\begin{tabular}{@{}c@{}}%
0\\0\\0
\end{tabular}\endgroup%
}}\!\right]$}%
}%
\ifdim\wd\matricesbox>\halfwidth\myboxwidth=\hsize\else\myboxwidth=\halfwidth\fi
\vbox{%
\ifdim\myboxwidth=\hsize
\setbox\onelinebox=\hbox{%
\vbox{\hbox{%
$\Pi_{19,4}$ spans $L_{16.13}$%
}\hbox{%
$36322|223636322\slashthree2236\rtimes D_{2}$%
}%
}%
\hfill\copy\matricesbox
}%
\ifdim\wd\onelinebox>\myboxwidth
\hbox to \myboxwidth{%
$\Pi_{19,4}$ spans $L_{16.13}$%
\hfil
$36322|223636322\slashthree2236\rtimes D_{2}$%
}%
\box\matricesbox
\else
\hbox to \myboxwidth{%
\unhbox\onelinebox
}%
\fi
\else
\hbox to \myboxwidth{%
$\Pi_{19,4}$ spans $L_{16.13}$%
\hfil}%
\hbox to \myboxwidth{%
$36322|223636322\slashthree2236\rtimes D_{2}$%
\hfil}%
\box\matricesbox
\fi
}%
\hfill\discretionary{}{}{}%
\setbox\matricesbox=\hbox{%
{$\left[\!\llap{\phantom{%
\begingroup \smaller\smaller\smaller\begin{tabular}{@{}c@{}}%
\phantom{0}\\\phantom{0}\\\phantom{0}
\end{tabular}\endgroup%
}}\right.$}%
\begingroup \smaller\smaller\smaller\begin{tabular}{@{}c@{}}%
-1\\\phantom{0}\\\phantom{0}
\end{tabular}\endgroup%
\kern3pt%
\begingroup \smaller\smaller\smaller\begin{tabular}{@{}c@{}}%
\phantom{0}\\45/2\\\phantom{0}
\end{tabular}\endgroup%
\kern3pt%
\begingroup \smaller\smaller\smaller\begin{tabular}{@{}c@{}}%
\phantom{0}\\\phantom{0}\\15/2
\end{tabular}\endgroup%
{$\left.\llap{\phantom{%
\begingroup \smaller\smaller\smaller\begin{tabular}{@{}c@{}}%
\phantom{0}\\\phantom{0}\\\phantom{0}
\end{tabular}\endgroup%
}}\!\right]$}%
{$\left[\!\llap{\phantom{%
\begingroup \smaller\smaller\smaller\begin{tabular}{@{}c@{}}%
0\\0\\0
\end{tabular}\endgroup%
}}\right.$}%
\begingroup \smaller\smaller\smaller\begin{tabular}{@{}c@{}}%
90\\19\\3
\end{tabular}\endgroup%
\kern3pt%
\begingroup \smaller\smaller\smaller\begin{tabular}{@{}c@{}}%
5\\1\\1
\end{tabular}\endgroup%
\kern3pt%
\begingroup \smaller\smaller\smaller\begin{tabular}{@{}c@{}}%
9\\1\\3
\end{tabular}\endgroup%
\kern3pt%
\begingroup \smaller\smaller\smaller\begin{tabular}{@{}c@{}}%
30\\1\\11
\end{tabular}\endgroup%
\kern3pt%
\begingroup \smaller\smaller\smaller\begin{tabular}{@{}c@{}}%
30\\-1\\11
\end{tabular}\endgroup%
\kern3pt%
\begingroup \smaller\smaller\smaller\begin{tabular}{@{}c@{}}%
90\\-8\\30
\end{tabular}\endgroup%
\kern3pt%
\begingroup \smaller\smaller\smaller\begin{tabular}{@{}c@{}}%
90\\-11\\27
\end{tabular}\endgroup%
\kern3pt%
\begingroup \smaller\smaller\smaller\begin{tabular}{@{}c@{}}%
30\\-5\\7
\end{tabular}\endgroup%
\kern3pt%
\begingroup \smaller\smaller\smaller\begin{tabular}{@{}c@{}}%
30\\-6\\4
\end{tabular}\endgroup%
\kern3pt%
\begingroup \smaller\smaller\smaller\begin{tabular}{@{}c@{}}%
9\\-2\\0
\end{tabular}\endgroup%
{$\left.\llap{\phantom{%
\begingroup \smaller\smaller\smaller\begin{tabular}{@{}c@{}}%
0\\0\\0
\end{tabular}\endgroup%
}}\!\right]$}%
}%
\ifdim\wd\matricesbox>\halfwidth\myboxwidth=\hsize\else\myboxwidth=\halfwidth\fi
\vbox{%
\ifdim\myboxwidth=\hsize
\setbox\onelinebox=\hbox{%
\vbox{\hbox{%
$\Pi_{19,5}$ spans $L_{16.13}$%
}\hbox{%
$363222\slashthree222363632|236\rtimes D_{2}$%
}%
}%
\hfill\copy\matricesbox
}%
\ifdim\wd\onelinebox>\myboxwidth
\hbox to \myboxwidth{%
$\Pi_{19,5}$ spans $L_{16.13}$%
\hfil
$363222\slashthree222363632|236\rtimes D_{2}$%
}%
\box\matricesbox
\else
\hbox to \myboxwidth{%
\unhbox\onelinebox
}%
\fi
\else
\hbox to \myboxwidth{%
$\Pi_{19,5}$ spans $L_{16.13}$%
\hfil}%
\hbox to \myboxwidth{%
$363222\slashthree222363632|236\rtimes D_{2}$%
\hfil}%
\box\matricesbox
\fi
}%
\hfill\discretionary{}{}{}%
\setbox\matricesbox=\hbox{%
{$\left[\!\llap{\phantom{%
\begingroup \smaller\smaller\smaller\begin{tabular}{@{}c@{}}%
\phantom{0}\\\phantom{0}\\\phantom{0}
\end{tabular}\endgroup%
}}\right.$}%
\begingroup \smaller\smaller\smaller\begin{tabular}{@{}c@{}}%
-1\\\phantom{0}\\\phantom{0}
\end{tabular}\endgroup%
\kern3pt%
\begingroup \smaller\smaller\smaller\begin{tabular}{@{}c@{}}%
\phantom{0}\\15/2\\\phantom{0}
\end{tabular}\endgroup%
\kern3pt%
\begingroup \smaller\smaller\smaller\begin{tabular}{@{}c@{}}%
\phantom{0}\\\phantom{0}\\45/2
\end{tabular}\endgroup%
{$\left.\llap{\phantom{%
\begingroup \smaller\smaller\smaller\begin{tabular}{@{}c@{}}%
\phantom{0}\\\phantom{0}\\\phantom{0}
\end{tabular}\endgroup%
}}\!\right]$}%
{$\left[\!\llap{\phantom{%
\begingroup \smaller\smaller\smaller\begin{tabular}{@{}c@{}}%
0\\0\\0
\end{tabular}\endgroup%
}}\right.$}%
\begingroup \smaller\smaller\smaller\begin{tabular}{@{}c@{}}%
5\\2\\0
\end{tabular}\endgroup%
\kern3pt%
\begingroup \smaller\smaller\smaller\begin{tabular}{@{}c@{}}%
90\\30\\-8
\end{tabular}\endgroup%
\kern3pt%
\begingroup \smaller\smaller\smaller\begin{tabular}{@{}c@{}}%
90\\27\\-11
\end{tabular}\endgroup%
\kern3pt%
\begingroup \smaller\smaller\smaller\begin{tabular}{@{}c@{}}%
5\\1\\-1
\end{tabular}\endgroup%
\kern3pt%
\begingroup \smaller\smaller\smaller\begin{tabular}{@{}c@{}}%
9\\0\\-2
\end{tabular}\endgroup%
\kern3pt%
\begingroup \smaller\smaller\smaller\begin{tabular}{@{}c@{}}%
30\\-4\\-6
\end{tabular}\endgroup%
\kern3pt%
\begingroup \smaller\smaller\smaller\begin{tabular}{@{}c@{}}%
30\\-7\\-5
\end{tabular}\endgroup%
\kern3pt%
\begingroup \smaller\smaller\smaller\begin{tabular}{@{}c@{}}%
90\\-27\\-11
\end{tabular}\endgroup%
\kern3pt%
\begingroup \smaller\smaller\smaller\begin{tabular}{@{}c@{}}%
90\\-30\\-8
\end{tabular}\endgroup%
\kern3pt%
\begingroup \smaller\smaller\smaller\begin{tabular}{@{}c@{}}%
30\\-11\\-1
\end{tabular}\endgroup%
{$\left.\llap{\phantom{%
\begingroup \smaller\smaller\smaller\begin{tabular}{@{}c@{}}%
0\\0\\0
\end{tabular}\endgroup%
}}\!\right]$}%
}%
\ifdim\wd\matricesbox>\halfwidth\myboxwidth=\hsize\else\myboxwidth=\halfwidth\fi
\vbox{%
\ifdim\myboxwidth=\hsize
\setbox\onelinebox=\hbox{%
\vbox{\hbox{%
$\Pi_{19,6}$ spans $L_{16.13}$%
}\hbox{%
$36322232|232223636\slashthree6\rtimes D_{2}$%
}%
}%
\hfill\copy\matricesbox
}%
\ifdim\wd\onelinebox>\myboxwidth
\hbox to \myboxwidth{%
$\Pi_{19,6}$ spans $L_{16.13}$%
\hfil
$36322232|232223636\slashthree6\rtimes D_{2}$%
}%
\box\matricesbox
\else
\hbox to \myboxwidth{%
\unhbox\onelinebox
}%
\fi
\else
\hbox to \myboxwidth{%
$\Pi_{19,6}$ spans $L_{16.13}$%
\hfil}%
\hbox to \myboxwidth{%
$36322232|232223636\slashthree6\rtimes D_{2}$%
\hfil}%
\box\matricesbox
\fi
}%
\hfill\discretionary{}{}{}%
\setbox\matricesbox=\hbox{%
{$\left[\!\llap{\phantom{%
\begingroup \smaller\smaller\smaller\begin{tabular}{@{}c@{}}%
\phantom{0}\\\phantom{0}\\\phantom{0}
\end{tabular}\endgroup%
}}\right.$}%
\begingroup \smaller\smaller\smaller\begin{tabular}{@{}c@{}}%
-1\\\phantom{0}\\\phantom{0}
\end{tabular}\endgroup%
\kern3pt%
\begingroup \smaller\smaller\smaller\begin{tabular}{@{}c@{}}%
\phantom{0}\\45/2\\\phantom{0}
\end{tabular}\endgroup%
\kern3pt%
\begingroup \smaller\smaller\smaller\begin{tabular}{@{}c@{}}%
\phantom{0}\\\phantom{0}\\15/2
\end{tabular}\endgroup%
{$\left.\llap{\phantom{%
\begingroup \smaller\smaller\smaller\begin{tabular}{@{}c@{}}%
\phantom{0}\\\phantom{0}\\\phantom{0}
\end{tabular}\endgroup%
}}\!\right]$}%
{$\left[\!\llap{\phantom{%
\begingroup \smaller\smaller\smaller\begin{tabular}{@{}c@{}}%
0\\0\\0
\end{tabular}\endgroup%
}}\right.$}%
\begingroup \smaller\smaller\smaller\begin{tabular}{@{}c@{}}%
90\\19\\3
\end{tabular}\endgroup%
\kern3pt%
\begingroup \smaller\smaller\smaller\begin{tabular}{@{}c@{}}%
30\\6\\4
\end{tabular}\endgroup%
\kern3pt%
\begingroup \smaller\smaller\smaller\begin{tabular}{@{}c@{}}%
30\\5\\7
\end{tabular}\endgroup%
\kern3pt%
\begingroup \smaller\smaller\smaller\begin{tabular}{@{}c@{}}%
9\\1\\3
\end{tabular}\endgroup%
\kern3pt%
\begingroup \smaller\smaller\smaller\begin{tabular}{@{}c@{}}%
5\\0\\2
\end{tabular}\endgroup%
\kern3pt%
\begingroup \smaller\smaller\smaller\begin{tabular}{@{}c@{}}%
90\\-8\\30
\end{tabular}\endgroup%
\kern3pt%
\begingroup \smaller\smaller\smaller\begin{tabular}{@{}c@{}}%
90\\-11\\27
\end{tabular}\endgroup%
\kern3pt%
\begingroup \smaller\smaller\smaller\begin{tabular}{@{}c@{}}%
30\\-5\\7
\end{tabular}\endgroup%
\kern3pt%
\begingroup \smaller\smaller\smaller\begin{tabular}{@{}c@{}}%
30\\-6\\4
\end{tabular}\endgroup%
\kern3pt%
\begingroup \smaller\smaller\smaller\begin{tabular}{@{}c@{}}%
9\\-2\\0
\end{tabular}\endgroup%
{$\left.\llap{\phantom{%
\begingroup \smaller\smaller\smaller\begin{tabular}{@{}c@{}}%
0\\0\\0
\end{tabular}\endgroup%
}}\!\right]$}%
}%
\ifdim\wd\matricesbox>\halfwidth\myboxwidth=\hsize\else\myboxwidth=\halfwidth\fi
\vbox{%
\ifdim\myboxwidth=\hsize
\setbox\onelinebox=\hbox{%
\vbox{\hbox{%
$\Pi_{19,7}$ spans $L_{16.13}$%
}\hbox{%
$\slashthree632223632|236322236\rtimes D_{2}$%
}%
}%
\hfill\copy\matricesbox
}%
\ifdim\wd\onelinebox>\myboxwidth
\hbox to \myboxwidth{%
$\Pi_{19,7}$ spans $L_{16.13}$%
\hfil
$\slashthree632223632|236322236\rtimes D_{2}$%
}%
\box\matricesbox
\else
\hbox to \myboxwidth{%
\unhbox\onelinebox
}%
\fi
\else
\hbox to \myboxwidth{%
$\Pi_{19,7}$ spans $L_{16.13}$%
\hfil}%
\hbox to \myboxwidth{%
$\slashthree632223632|236322236\rtimes D_{2}$%
\hfil}%
\box\matricesbox
\fi
}%
\hfill\discretionary{}{}{}%
\setbox\matricesbox=\hbox{%
{$\left[\!\llap{\phantom{%
\begingroup \smaller\smaller\smaller\begin{tabular}{@{}c@{}}%
\phantom{0}\\\phantom{0}\\\phantom{0}
\end{tabular}\endgroup%
}}\right.$}%
\begingroup \smaller\smaller\smaller\begin{tabular}{@{}c@{}}%
-1\\\phantom{0}\\\phantom{0}
\end{tabular}\endgroup%
\kern3pt%
\begingroup \smaller\smaller\smaller\begin{tabular}{@{}c@{}}%
\phantom{0}\\15/2\\\phantom{0}
\end{tabular}\endgroup%
\kern3pt%
\begingroup \smaller\smaller\smaller\begin{tabular}{@{}c@{}}%
\phantom{0}\\\phantom{0}\\45/2
\end{tabular}\endgroup%
{$\left.\llap{\phantom{%
\begingroup \smaller\smaller\smaller\begin{tabular}{@{}c@{}}%
\phantom{0}\\\phantom{0}\\\phantom{0}
\end{tabular}\endgroup%
}}\!\right]$}%
{$\left[\!\llap{\phantom{%
\begingroup \smaller\smaller\smaller\begin{tabular}{@{}c@{}}%
0\\0\\0
\end{tabular}\endgroup%
}}\right.$}%
\begingroup \smaller\smaller\smaller\begin{tabular}{@{}c@{}}%
30\\11\\1
\end{tabular}\endgroup%
\kern3pt%
\begingroup \smaller\smaller\smaller\begin{tabular}{@{}c@{}}%
90\\30\\8
\end{tabular}\endgroup%
\kern3pt%
\begingroup \smaller\smaller\smaller\begin{tabular}{@{}c@{}}%
90\\27\\11
\end{tabular}\endgroup%
\kern3pt%
\begingroup \smaller\smaller\smaller\begin{tabular}{@{}c@{}}%
5\\1\\1
\end{tabular}\endgroup%
\kern3pt%
\begingroup \smaller\smaller\smaller\begin{tabular}{@{}c@{}}%
9\\0\\2
\end{tabular}\endgroup%
\kern3pt%
\begingroup \smaller\smaller\smaller\begin{tabular}{@{}c@{}}%
30\\-4\\6
\end{tabular}\endgroup%
\kern3pt%
\begingroup \smaller\smaller\smaller\begin{tabular}{@{}c@{}}%
30\\-7\\5
\end{tabular}\endgroup%
\kern3pt%
\begingroup \smaller\smaller\smaller\begin{tabular}{@{}c@{}}%
90\\-27\\11
\end{tabular}\endgroup%
\kern3pt%
\begingroup \smaller\smaller\smaller\begin{tabular}{@{}c@{}}%
90\\-30\\8
\end{tabular}\endgroup%
\kern3pt%
\begingroup \smaller\smaller\smaller\begin{tabular}{@{}c@{}}%
5\\-2\\0
\end{tabular}\endgroup%
{$\left.\llap{\phantom{%
\begingroup \smaller\smaller\smaller\begin{tabular}{@{}c@{}}%
0\\0\\0
\end{tabular}\endgroup%
}}\!\right]$}%
}%
\ifdim\wd\matricesbox>\halfwidth\myboxwidth=\hsize\else\myboxwidth=\halfwidth\fi
\vbox{%
\ifdim\myboxwidth=\hsize
\setbox\onelinebox=\hbox{%
\vbox{\hbox{%
$\Pi_{19,8}$ spans $L_{16.13}$%
}\hbox{%
$36322236\slashthree632223632|2\rtimes D_{2}$%
}%
}%
\hfill\copy\matricesbox
}%
\ifdim\wd\onelinebox>\myboxwidth
\hbox to \myboxwidth{%
$\Pi_{19,8}$ spans $L_{16.13}$%
\hfil
$36322236\slashthree632223632|2\rtimes D_{2}$%
}%
\box\matricesbox
\else
\hbox to \myboxwidth{%
\unhbox\onelinebox
}%
\fi
\else
\hbox to \myboxwidth{%
$\Pi_{19,8}$ spans $L_{16.13}$%
\hfil}%
\hbox to \myboxwidth{%
$36322236\slashthree632223632|2\rtimes D_{2}$%
\hfil}%
\box\matricesbox
\fi
}%
\hfill\discretionary{}{}{}%
\setbox\matricesbox=\hbox{%
{$\left[\!\llap{\phantom{%
\begingroup \smaller\smaller\smaller\begin{tabular}{@{}c@{}}%
\phantom{0}\\\phantom{0}\\\phantom{0}
\end{tabular}\endgroup%
}}\right.$}%
\begingroup \smaller\smaller\smaller\begin{tabular}{@{}c@{}}%
-1\\\phantom{0}\\\phantom{0}
\end{tabular}\endgroup%
\kern3pt%
\begingroup \smaller\smaller\smaller\begin{tabular}{@{}c@{}}%
\phantom{0}\\15/2\\\phantom{0}
\end{tabular}\endgroup%
\kern3pt%
\begingroup \smaller\smaller\smaller\begin{tabular}{@{}c@{}}%
\phantom{0}\\\phantom{0}\\45/2
\end{tabular}\endgroup%
{$\left.\llap{\phantom{%
\begingroup \smaller\smaller\smaller\begin{tabular}{@{}c@{}}%
\phantom{0}\\\phantom{0}\\\phantom{0}
\end{tabular}\endgroup%
}}\!\right]$}%
{$\left[\!\llap{\phantom{%
\begingroup \smaller\smaller\smaller\begin{tabular}{@{}c@{}}%
0\\0\\0
\end{tabular}\endgroup%
}}\right.$}%
\begingroup \smaller\smaller\smaller\begin{tabular}{@{}c@{}}%
5\\2\\0
\end{tabular}\endgroup%
\kern3pt%
\begingroup \smaller\smaller\smaller\begin{tabular}{@{}c@{}}%
9\\3\\-1
\end{tabular}\endgroup%
\kern3pt%
\begingroup \smaller\smaller\smaller\begin{tabular}{@{}c@{}}%
30\\7\\-5
\end{tabular}\endgroup%
\kern3pt%
\begingroup \smaller\smaller\smaller\begin{tabular}{@{}c@{}}%
30\\4\\-6
\end{tabular}\endgroup%
\kern3pt%
\begingroup \smaller\smaller\smaller\begin{tabular}{@{}c@{}}%
9\\0\\-2
\end{tabular}\endgroup%
\kern3pt%
\begingroup \smaller\smaller\smaller\begin{tabular}{@{}c@{}}%
30\\-4\\-6
\end{tabular}\endgroup%
\kern3pt%
\begingroup \smaller\smaller\smaller\begin{tabular}{@{}c@{}}%
30\\-7\\-5
\end{tabular}\endgroup%
\kern3pt%
\begingroup \smaller\smaller\smaller\begin{tabular}{@{}c@{}}%
90\\-27\\-11
\end{tabular}\endgroup%
\kern3pt%
\begingroup \smaller\smaller\smaller\begin{tabular}{@{}c@{}}%
90\\-30\\-8
\end{tabular}\endgroup%
\kern3pt%
\begingroup \smaller\smaller\smaller\begin{tabular}{@{}c@{}}%
30\\-11\\-1
\end{tabular}\endgroup%
{$\left.\llap{\phantom{%
\begingroup \smaller\smaller\smaller\begin{tabular}{@{}c@{}}%
0\\0\\0
\end{tabular}\endgroup%
}}\!\right]$}%
}%
\ifdim\wd\matricesbox>\halfwidth\myboxwidth=\hsize\else\myboxwidth=\halfwidth\fi
\vbox{%
\ifdim\myboxwidth=\hsize
\setbox\onelinebox=\hbox{%
\vbox{\hbox{%
$\Pi_{19,9}$ spans $L_{16.13}$%
}\hbox{%
$36322322|223223636\slashthree6\rtimes D_{2}$%
}%
}%
\hfill\copy\matricesbox
}%
\ifdim\wd\onelinebox>\myboxwidth
\hbox to \myboxwidth{%
$\Pi_{19,9}$ spans $L_{16.13}$%
\hfil
$36322322|223223636\slashthree6\rtimes D_{2}$%
}%
\box\matricesbox
\else
\hbox to \myboxwidth{%
\unhbox\onelinebox
}%
\fi
\else
\hbox to \myboxwidth{%
$\Pi_{19,9}$ spans $L_{16.13}$%
\hfil}%
\hbox to \myboxwidth{%
$36322322|223223636\slashthree6\rtimes D_{2}$%
\hfil}%
\box\matricesbox
\fi
}%
\hfill\discretionary{}{}{}%
\setbox\matricesbox=\hbox{%
{$\left[\!\llap{\phantom{%
\begingroup \smaller\smaller\smaller\begin{tabular}{@{}c@{}}%
\phantom{0}\\\phantom{0}\\\phantom{0}
\end{tabular}\endgroup%
}}\right.$}%
\begingroup \smaller\smaller\smaller\begin{tabular}{@{}c@{}}%
-3\\\phantom{0}\\\phantom{0}
\end{tabular}\endgroup%
\kern3pt%
\begingroup \smaller\smaller\smaller\begin{tabular}{@{}c@{}}%
\phantom{0}\\15/2\\\phantom{0}
\end{tabular}\endgroup%
\kern3pt%
\begingroup \smaller\smaller\smaller\begin{tabular}{@{}c@{}}%
\phantom{0}\\\phantom{0}\\5/2
\end{tabular}\endgroup%
{$\left.\llap{\phantom{%
\begingroup \smaller\smaller\smaller\begin{tabular}{@{}c@{}}%
\phantom{0}\\\phantom{0}\\\phantom{0}
\end{tabular}\endgroup%
}}\!\right]$}%
{$\left[\!\llap{\phantom{%
\begingroup \smaller\smaller\smaller\begin{tabular}{@{}c@{}}%
0\\0\\0
\end{tabular}\endgroup%
}}\right.$}%
\begingroup \smaller\smaller\smaller\begin{tabular}{@{}c@{}}%
30\\-19\\3
\end{tabular}\endgroup%
\kern3pt%
\begingroup \smaller\smaller\smaller\begin{tabular}{@{}c@{}}%
10\\-6\\4
\end{tabular}\endgroup%
\kern3pt%
\begingroup \smaller\smaller\smaller\begin{tabular}{@{}c@{}}%
10\\-5\\7
\end{tabular}\endgroup%
\kern3pt%
\begingroup \smaller\smaller\smaller\begin{tabular}{@{}c@{}}%
3\\-1\\3
\end{tabular}\endgroup%
\kern3pt%
\begingroup \smaller\smaller\smaller\begin{tabular}{@{}c@{}}%
10\\-1\\11
\end{tabular}\endgroup%
\kern3pt%
\begingroup \smaller\smaller\smaller\begin{tabular}{@{}c@{}}%
10\\1\\11
\end{tabular}\endgroup%
\kern3pt%
\begingroup \smaller\smaller\smaller\begin{tabular}{@{}c@{}}%
3\\1\\3
\end{tabular}\endgroup%
\kern3pt%
\begingroup \smaller\smaller\smaller\begin{tabular}{@{}c@{}}%
10\\5\\7
\end{tabular}\endgroup%
\kern3pt%
\begingroup \smaller\smaller\smaller\begin{tabular}{@{}c@{}}%
10\\6\\4
\end{tabular}\endgroup%
\kern3pt%
\begingroup \smaller\smaller\smaller\begin{tabular}{@{}c@{}}%
3\\2\\0
\end{tabular}\endgroup%
{$\left.\llap{\phantom{%
\begingroup \smaller\smaller\smaller\begin{tabular}{@{}c@{}}%
0\\0\\0
\end{tabular}\endgroup%
}}\!\right]$}%
}%
\ifdim\wd\matricesbox>\halfwidth\myboxwidth=\hsize\else\myboxwidth=\halfwidth\fi
\vbox{%
\ifdim\myboxwidth=\hsize
\setbox\onelinebox=\hbox{%
\vbox{\hbox{%
$\Pi_{19,10}$ spans $L_{16.9}$%
}\hbox{%
$\slashthree632232232|232232236\rtimes D_{2}$%
}%
}%
\hfill\copy\matricesbox
}%
\ifdim\wd\onelinebox>\myboxwidth
\hbox to \myboxwidth{%
$\Pi_{19,10}$ spans $L_{16.9}$%
\hfil
$\slashthree632232232|232232236\rtimes D_{2}$%
}%
\box\matricesbox
\else
\hbox to \myboxwidth{%
\unhbox\onelinebox
}%
\fi
\else
\hbox to \myboxwidth{%
$\Pi_{19,10}$ spans $L_{16.9}$%
\hfil}%
\hbox to \myboxwidth{%
$\slashthree632232232|232232236\rtimes D_{2}$%
\hfil}%
\box\matricesbox
\fi
}%
\hfill\discretionary{}{}{}%
\setbox\matricesbox=\hbox{%
{$\left[\!\llap{\phantom{%
\begingroup \smaller\smaller\smaller\begin{tabular}{@{}c@{}}%
\phantom{0}\\\phantom{0}\\\phantom{0}
\end{tabular}\endgroup%
}}\right.$}%
\begingroup \smaller\smaller\smaller\begin{tabular}{@{}c@{}}%
-1\\\phantom{0}\\\phantom{0}
\end{tabular}\endgroup%
\kern3pt%
\begingroup \smaller\smaller\smaller\begin{tabular}{@{}c@{}}%
\phantom{0}\\15/2\\\phantom{0}
\end{tabular}\endgroup%
\kern3pt%
\begingroup \smaller\smaller\smaller\begin{tabular}{@{}c@{}}%
\phantom{0}\\\phantom{0}\\45/2
\end{tabular}\endgroup%
{$\left.\llap{\phantom{%
\begingroup \smaller\smaller\smaller\begin{tabular}{@{}c@{}}%
\phantom{0}\\\phantom{0}\\\phantom{0}
\end{tabular}\endgroup%
}}\!\right]$}%
{$\left[\!\llap{\phantom{%
\begingroup \smaller\smaller\smaller\begin{tabular}{@{}c@{}}%
0\\0\\0
\end{tabular}\endgroup%
}}\right.$}%
\begingroup \smaller\smaller\smaller\begin{tabular}{@{}c@{}}%
30\\11\\1
\end{tabular}\endgroup%
\kern3pt%
\begingroup \smaller\smaller\smaller\begin{tabular}{@{}c@{}}%
9\\3\\1
\end{tabular}\endgroup%
\kern3pt%
\begingroup \smaller\smaller\smaller\begin{tabular}{@{}c@{}}%
30\\7\\5
\end{tabular}\endgroup%
\kern3pt%
\begingroup \smaller\smaller\smaller\begin{tabular}{@{}c@{}}%
30\\4\\6
\end{tabular}\endgroup%
\kern3pt%
\begingroup \smaller\smaller\smaller\begin{tabular}{@{}c@{}}%
9\\0\\2
\end{tabular}\endgroup%
\kern3pt%
\begingroup \smaller\smaller\smaller\begin{tabular}{@{}c@{}}%
30\\-4\\6
\end{tabular}\endgroup%
\kern3pt%
\begingroup \smaller\smaller\smaller\begin{tabular}{@{}c@{}}%
30\\-7\\5
\end{tabular}\endgroup%
\kern3pt%
\begingroup \smaller\smaller\smaller\begin{tabular}{@{}c@{}}%
90\\-27\\11
\end{tabular}\endgroup%
\kern3pt%
\begingroup \smaller\smaller\smaller\begin{tabular}{@{}c@{}}%
90\\-30\\8
\end{tabular}\endgroup%
\kern3pt%
\begingroup \smaller\smaller\smaller\begin{tabular}{@{}c@{}}%
5\\-2\\0
\end{tabular}\endgroup%
{$\left.\llap{\phantom{%
\begingroup \smaller\smaller\smaller\begin{tabular}{@{}c@{}}%
0\\0\\0
\end{tabular}\endgroup%
}}\!\right]$}%
}%
\ifdim\wd\matricesbox>\halfwidth\myboxwidth=\hsize\else\myboxwidth=\halfwidth\fi
\vbox{%
\ifdim\myboxwidth=\hsize
\setbox\onelinebox=\hbox{%
\vbox{\hbox{%
$\Pi_{19,11}$ spans $L_{16.13}$%
}\hbox{%
$36322322\slashthree223223632|2\rtimes D_{2}$%
}%
}%
\hfill\copy\matricesbox
}%
\ifdim\wd\onelinebox>\myboxwidth
\hbox to \myboxwidth{%
$\Pi_{19,11}$ spans $L_{16.13}$%
\hfil
$36322322\slashthree223223632|2\rtimes D_{2}$%
}%
\box\matricesbox
\else
\hbox to \myboxwidth{%
\unhbox\onelinebox
}%
\fi
\else
\hbox to \myboxwidth{%
$\Pi_{19,11}$ spans $L_{16.13}$%
\hfil}%
\hbox to \myboxwidth{%
$36322322\slashthree223223632|2\rtimes D_{2}$%
\hfil}%
\box\matricesbox
\fi
}%
\hfill\discretionary{}{}{}%
\setbox\matricesbox=\hbox{%
{$\left[\!\llap{\phantom{%
\begingroup \smaller\smaller\smaller\begin{tabular}{@{}c@{}}%
\phantom{0}\\\phantom{0}\\\phantom{0}
\end{tabular}\endgroup%
}}\right.$}%
\begingroup \smaller\smaller\smaller\begin{tabular}{@{}c@{}}%
-1\\\phantom{0}\\\phantom{0}
\end{tabular}\endgroup%
\kern3pt%
\begingroup \smaller\smaller\smaller\begin{tabular}{@{}c@{}}%
\phantom{0}\\45/2\\\phantom{0}
\end{tabular}\endgroup%
\kern3pt%
\begingroup \smaller\smaller\smaller\begin{tabular}{@{}c@{}}%
\phantom{0}\\\phantom{0}\\15/2
\end{tabular}\endgroup%
{$\left.\llap{\phantom{%
\begingroup \smaller\smaller\smaller\begin{tabular}{@{}c@{}}%
\phantom{0}\\\phantom{0}\\\phantom{0}
\end{tabular}\endgroup%
}}\!\right]$}%
{$\left[\!\llap{\phantom{%
\begingroup \smaller\smaller\smaller\begin{tabular}{@{}c@{}}%
0\\0\\0
\end{tabular}\endgroup%
}}\right.$}%
\begingroup \smaller\smaller\smaller\begin{tabular}{@{}c@{}}%
9\\2\\0
\end{tabular}\endgroup%
\kern3pt%
\begingroup \smaller\smaller\smaller\begin{tabular}{@{}c@{}}%
30\\6\\4
\end{tabular}\endgroup%
\kern3pt%
\begingroup \smaller\smaller\smaller\begin{tabular}{@{}c@{}}%
30\\5\\7
\end{tabular}\endgroup%
\kern3pt%
\begingroup \smaller\smaller\smaller\begin{tabular}{@{}c@{}}%
9\\1\\3
\end{tabular}\endgroup%
\kern3pt%
\begingroup \smaller\smaller\smaller\begin{tabular}{@{}c@{}}%
30\\1\\11
\end{tabular}\endgroup%
\kern3pt%
\begingroup \smaller\smaller\smaller\begin{tabular}{@{}c@{}}%
30\\-1\\11
\end{tabular}\endgroup%
\kern3pt%
\begingroup \smaller\smaller\smaller\begin{tabular}{@{}c@{}}%
90\\-8\\30
\end{tabular}\endgroup%
\kern3pt%
\begingroup \smaller\smaller\smaller\begin{tabular}{@{}c@{}}%
90\\-11\\27
\end{tabular}\endgroup%
\kern3pt%
\begingroup \smaller\smaller\smaller\begin{tabular}{@{}c@{}}%
5\\-1\\1
\end{tabular}\endgroup%
\kern3pt%
\begingroup \smaller\smaller\smaller\begin{tabular}{@{}c@{}}%
90\\-19\\3
\end{tabular}\endgroup%
{$\left.\llap{\phantom{%
\begingroup \smaller\smaller\smaller\begin{tabular}{@{}c@{}}%
0\\0\\0
\end{tabular}\endgroup%
}}\!\right]$}%
}%
\ifdim\wd\matricesbox>\halfwidth\myboxwidth=\hsize\else\myboxwidth=\halfwidth\fi
\vbox{%
\ifdim\myboxwidth=\hsize
\setbox\onelinebox=\hbox{%
\vbox{\hbox{%
$\Pi_{19,12}$ spans $L_{16.13}$%
}\hbox{%
$3632232|232236322\slashthree22\rtimes D_{2}$%
}%
}%
\hfill\copy\matricesbox
}%
\ifdim\wd\onelinebox>\myboxwidth
\hbox to \myboxwidth{%
$\Pi_{19,12}$ spans $L_{16.13}$%
\hfil
$3632232|232236322\slashthree22\rtimes D_{2}$%
}%
\box\matricesbox
\else
\hbox to \myboxwidth{%
\unhbox\onelinebox
}%
\fi
\else
\hbox to \myboxwidth{%
$\Pi_{19,12}$ spans $L_{16.13}$%
\hfil}%
\hbox to \myboxwidth{%
$3632232|232236322\slashthree22\rtimes D_{2}$%
\hfil}%
\box\matricesbox
\fi
}%
\hfill\discretionary{}{}{}%
\setbox\matricesbox=\hbox{%
{$\left[\!\llap{\phantom{%
\begingroup \smaller\smaller\smaller\begin{tabular}{@{}c@{}}%
\phantom{0}\\\phantom{0}\\\phantom{0}
\end{tabular}\endgroup%
}}\right.$}%
\begingroup \smaller\smaller\smaller\begin{tabular}{@{}c@{}}%
-1\\\phantom{0}\\\phantom{0}
\end{tabular}\endgroup%
\kern3pt%
\begingroup \smaller\smaller\smaller\begin{tabular}{@{}c@{}}%
\phantom{0}\\15/2\\\phantom{0}
\end{tabular}\endgroup%
\kern3pt%
\begingroup \smaller\smaller\smaller\begin{tabular}{@{}c@{}}%
\phantom{0}\\\phantom{0}\\45/2
\end{tabular}\endgroup%
{$\left.\llap{\phantom{%
\begingroup \smaller\smaller\smaller\begin{tabular}{@{}c@{}}%
\phantom{0}\\\phantom{0}\\\phantom{0}
\end{tabular}\endgroup%
}}\!\right]$}%
{$\left[\!\llap{\phantom{%
\begingroup \smaller\smaller\smaller\begin{tabular}{@{}c@{}}%
0\\0\\0
\end{tabular}\endgroup%
}}\right.$}%
\begingroup \smaller\smaller\smaller\begin{tabular}{@{}c@{}}%
30\\-11\\-1
\end{tabular}\endgroup%
\kern3pt%
\begingroup \smaller\smaller\smaller\begin{tabular}{@{}c@{}}%
9\\-3\\-1
\end{tabular}\endgroup%
\kern3pt%
\begingroup \smaller\smaller\smaller\begin{tabular}{@{}c@{}}%
30\\-7\\-5
\end{tabular}\endgroup%
\kern3pt%
\begingroup \smaller\smaller\smaller\begin{tabular}{@{}c@{}}%
30\\-4\\-6
\end{tabular}\endgroup%
\kern3pt%
\begingroup \smaller\smaller\smaller\begin{tabular}{@{}c@{}}%
90\\-3\\-19
\end{tabular}\endgroup%
\kern3pt%
\begingroup \smaller\smaller\smaller\begin{tabular}{@{}c@{}}%
90\\3\\-19
\end{tabular}\endgroup%
\kern3pt%
\begingroup \smaller\smaller\smaller\begin{tabular}{@{}c@{}}%
5\\1\\-1
\end{tabular}\endgroup%
\kern3pt%
\begingroup \smaller\smaller\smaller\begin{tabular}{@{}c@{}}%
90\\27\\-11
\end{tabular}\endgroup%
\kern3pt%
\begingroup \smaller\smaller\smaller\begin{tabular}{@{}c@{}}%
90\\30\\-8
\end{tabular}\endgroup%
\kern3pt%
\begingroup \smaller\smaller\smaller\begin{tabular}{@{}c@{}}%
5\\2\\0
\end{tabular}\endgroup%
{$\left.\llap{\phantom{%
\begingroup \smaller\smaller\smaller\begin{tabular}{@{}c@{}}%
0\\0\\0
\end{tabular}\endgroup%
}}\!\right]$}%
}%
\ifdim\wd\matricesbox>\halfwidth\myboxwidth=\hsize\else\myboxwidth=\halfwidth\fi
\vbox{%
\ifdim\myboxwidth=\hsize
\setbox\onelinebox=\hbox{%
\vbox{\hbox{%
$\Pi_{19,13}$ spans $L_{16.13}$%
}\hbox{%
$36322\slashthree223632232|2322\rtimes D_{2}$%
}%
}%
\hfill\copy\matricesbox
}%
\ifdim\wd\onelinebox>\myboxwidth
\hbox to \myboxwidth{%
$\Pi_{19,13}$ spans $L_{16.13}$%
\hfil
$36322\slashthree223632232|2322\rtimes D_{2}$%
}%
\box\matricesbox
\else
\hbox to \myboxwidth{%
\unhbox\onelinebox
}%
\fi
\else
\hbox to \myboxwidth{%
$\Pi_{19,13}$ spans $L_{16.13}$%
\hfil}%
\hbox to \myboxwidth{%
$36322\slashthree223632232|2322\rtimes D_{2}$%
\hfil}%
\box\matricesbox
\fi
}%
\hfill\discretionary{}{}{}%
\setbox\matricesbox=\hbox{%
{$\left[\!\llap{\phantom{%
\begingroup \smaller\smaller\smaller\begin{tabular}{@{}c@{}}%
\phantom{0}\\\phantom{0}\\\phantom{0}
\end{tabular}\endgroup%
}}\right.$}%
\begingroup \smaller\smaller\smaller\begin{tabular}{@{}c@{}}%
-1\\\phantom{0}\\\phantom{0}
\end{tabular}\endgroup%
\kern3pt%
\begingroup \smaller\smaller\smaller\begin{tabular}{@{}c@{}}%
\phantom{0}\\45/2\\\phantom{0}
\end{tabular}\endgroup%
\kern3pt%
\begingroup \smaller\smaller\smaller\begin{tabular}{@{}c@{}}%
\phantom{0}\\\phantom{0}\\15/2
\end{tabular}\endgroup%
{$\left.\llap{\phantom{%
\begingroup \smaller\smaller\smaller\begin{tabular}{@{}c@{}}%
\phantom{0}\\\phantom{0}\\\phantom{0}
\end{tabular}\endgroup%
}}\!\right]$}%
{$\left[\!\llap{\phantom{%
\begingroup \smaller\smaller\smaller\begin{tabular}{@{}c@{}}%
0\\0\\0
\end{tabular}\endgroup%
}}\right.$}%
\begingroup \smaller\smaller\smaller\begin{tabular}{@{}c@{}}%
90\\19\\3
\end{tabular}\endgroup%
\kern3pt%
\begingroup \smaller\smaller\smaller\begin{tabular}{@{}c@{}}%
30\\6\\4
\end{tabular}\endgroup%
\kern3pt%
\begingroup \smaller\smaller\smaller\begin{tabular}{@{}c@{}}%
30\\5\\7
\end{tabular}\endgroup%
\kern3pt%
\begingroup \smaller\smaller\smaller\begin{tabular}{@{}c@{}}%
9\\1\\3
\end{tabular}\endgroup%
\kern3pt%
\begingroup \smaller\smaller\smaller\begin{tabular}{@{}c@{}}%
30\\1\\11
\end{tabular}\endgroup%
\kern3pt%
\begingroup \smaller\smaller\smaller\begin{tabular}{@{}c@{}}%
30\\-1\\11
\end{tabular}\endgroup%
\kern3pt%
\begingroup \smaller\smaller\smaller\begin{tabular}{@{}c@{}}%
90\\-8\\30
\end{tabular}\endgroup%
\kern3pt%
\begingroup \smaller\smaller\smaller\begin{tabular}{@{}c@{}}%
90\\-11\\27
\end{tabular}\endgroup%
\kern3pt%
\begingroup \smaller\smaller\smaller\begin{tabular}{@{}c@{}}%
5\\-1\\1
\end{tabular}\endgroup%
\kern3pt%
\begingroup \smaller\smaller\smaller\begin{tabular}{@{}c@{}}%
9\\-2\\0
\end{tabular}\endgroup%
{$\left.\llap{\phantom{%
\begingroup \smaller\smaller\smaller\begin{tabular}{@{}c@{}}%
0\\0\\0
\end{tabular}\endgroup%
}}\!\right]$}%
}%
\ifdim\wd\matricesbox>\halfwidth\myboxwidth=\hsize\else\myboxwidth=\halfwidth\fi
\vbox{%
\ifdim\myboxwidth=\hsize
\setbox\onelinebox=\hbox{%
\vbox{\hbox{%
$\Pi_{19,14}$ spans $L_{16.13}$%
}\hbox{%
$\slashthree632236322|223632236\rtimes D_{2}$%
}%
}%
\hfill\copy\matricesbox
}%
\ifdim\wd\onelinebox>\myboxwidth
\hbox to \myboxwidth{%
$\Pi_{19,14}$ spans $L_{16.13}$%
\hfil
$\slashthree632236322|223632236\rtimes D_{2}$%
}%
\box\matricesbox
\else
\hbox to \myboxwidth{%
\unhbox\onelinebox
}%
\fi
\else
\hbox to \myboxwidth{%
$\Pi_{19,14}$ spans $L_{16.13}$%
\hfil}%
\hbox to \myboxwidth{%
$\slashthree632236322|223632236\rtimes D_{2}$%
\hfil}%
\box\matricesbox
\fi
}%
\hfill\discretionary{}{}{}%
\setbox\matricesbox=\hbox{%
{$\left[\!\llap{\phantom{%
\begingroup \smaller\smaller\smaller\begin{tabular}{@{}c@{}}%
\phantom{0}\\\phantom{0}\\\phantom{0}
\end{tabular}\endgroup%
}}\right.$}%
\begingroup \smaller\smaller\smaller\begin{tabular}{@{}c@{}}%
-1\\\phantom{0}\\\phantom{0}
\end{tabular}\endgroup%
\kern3pt%
\begingroup \smaller\smaller\smaller\begin{tabular}{@{}c@{}}%
\phantom{0}\\45/2\\\phantom{0}
\end{tabular}\endgroup%
\kern3pt%
\begingroup \smaller\smaller\smaller\begin{tabular}{@{}c@{}}%
\phantom{0}\\\phantom{0}\\15/2
\end{tabular}\endgroup%
{$\left.\llap{\phantom{%
\begingroup \smaller\smaller\smaller\begin{tabular}{@{}c@{}}%
\phantom{0}\\\phantom{0}\\\phantom{0}
\end{tabular}\endgroup%
}}\!\right]$}%
{$\left[\!\llap{\phantom{%
\begingroup \smaller\smaller\smaller\begin{tabular}{@{}c@{}}%
0\\0\\0
\end{tabular}\endgroup%
}}\right.$}%
\begingroup \smaller\smaller\smaller\begin{tabular}{@{}c@{}}%
9\\-2\\0
\end{tabular}\endgroup%
\kern3pt%
\begingroup \smaller\smaller\smaller\begin{tabular}{@{}c@{}}%
30\\-6\\4
\end{tabular}\endgroup%
\kern3pt%
\begingroup \smaller\smaller\smaller\begin{tabular}{@{}c@{}}%
30\\-5\\7
\end{tabular}\endgroup%
\kern3pt%
\begingroup \smaller\smaller\smaller\begin{tabular}{@{}c@{}}%
90\\-11\\27
\end{tabular}\endgroup%
\kern3pt%
\begingroup \smaller\smaller\smaller\begin{tabular}{@{}c@{}}%
90\\-8\\30
\end{tabular}\endgroup%
\kern3pt%
\begingroup \smaller\smaller\smaller\begin{tabular}{@{}c@{}}%
5\\0\\2
\end{tabular}\endgroup%
\kern3pt%
\begingroup \smaller\smaller\smaller\begin{tabular}{@{}c@{}}%
90\\8\\30
\end{tabular}\endgroup%
\kern3pt%
\begingroup \smaller\smaller\smaller\begin{tabular}{@{}c@{}}%
90\\11\\27
\end{tabular}\endgroup%
\kern3pt%
\begingroup \smaller\smaller\smaller\begin{tabular}{@{}c@{}}%
5\\1\\1
\end{tabular}\endgroup%
\kern3pt%
\begingroup \smaller\smaller\smaller\begin{tabular}{@{}c@{}}%
90\\19\\3
\end{tabular}\endgroup%
{$\left.\llap{\phantom{%
\begingroup \smaller\smaller\smaller\begin{tabular}{@{}c@{}}%
0\\0\\0
\end{tabular}\endgroup%
}}\!\right]$}%
}%
\ifdim\wd\matricesbox>\halfwidth\myboxwidth=\hsize\else\myboxwidth=\halfwidth\fi
\vbox{%
\ifdim\myboxwidth=\hsize
\setbox\onelinebox=\hbox{%
\vbox{\hbox{%
$\Pi_{19,15}$ spans $L_{16.13}$%
}\hbox{%
$3632|236322322\slashthree22322\rtimes D_{2}$%
}%
}%
\hfill\copy\matricesbox
}%
\ifdim\wd\onelinebox>\myboxwidth
\hbox to \myboxwidth{%
$\Pi_{19,15}$ spans $L_{16.13}$%
\hfil
$3632|236322322\slashthree22322\rtimes D_{2}$%
}%
\box\matricesbox
\else
\hbox to \myboxwidth{%
\unhbox\onelinebox
}%
\fi
\else
\hbox to \myboxwidth{%
$\Pi_{19,15}$ spans $L_{16.13}$%
\hfil}%
\hbox to \myboxwidth{%
$3632|236322322\slashthree22322\rtimes D_{2}$%
\hfil}%
\box\matricesbox
\fi
}%
\hfill\discretionary{}{}{}%
\setbox\matricesbox=\hbox{%
{$\left[\!\llap{\phantom{%
\begingroup \smaller\smaller\smaller\begin{tabular}{@{}c@{}}%
\phantom{0}\\\phantom{0}\\\phantom{0}
\end{tabular}\endgroup%
}}\right.$}%
\begingroup \smaller\smaller\smaller\begin{tabular}{@{}c@{}}%
-1\\\phantom{0}\\\phantom{0}
\end{tabular}\endgroup%
\kern3pt%
\begingroup \smaller\smaller\smaller\begin{tabular}{@{}c@{}}%
\phantom{0}\\45/2\\\phantom{0}
\end{tabular}\endgroup%
\kern3pt%
\begingroup \smaller\smaller\smaller\begin{tabular}{@{}c@{}}%
\phantom{0}\\\phantom{0}\\15/2
\end{tabular}\endgroup%
{$\left.\llap{\phantom{%
\begingroup \smaller\smaller\smaller\begin{tabular}{@{}c@{}}%
\phantom{0}\\\phantom{0}\\\phantom{0}
\end{tabular}\endgroup%
}}\!\right]$}%
{$\left[\!\llap{\phantom{%
\begingroup \smaller\smaller\smaller\begin{tabular}{@{}c@{}}%
0\\0\\0
\end{tabular}\endgroup%
}}\right.$}%
\begingroup \smaller\smaller\smaller\begin{tabular}{@{}c@{}}%
9\\2\\0
\end{tabular}\endgroup%
\kern3pt%
\begingroup \smaller\smaller\smaller\begin{tabular}{@{}c@{}}%
5\\1\\1
\end{tabular}\endgroup%
\kern3pt%
\begingroup \smaller\smaller\smaller\begin{tabular}{@{}c@{}}%
90\\11\\27
\end{tabular}\endgroup%
\kern3pt%
\begingroup \smaller\smaller\smaller\begin{tabular}{@{}c@{}}%
90\\8\\30
\end{tabular}\endgroup%
\kern3pt%
\begingroup \smaller\smaller\smaller\begin{tabular}{@{}c@{}}%
30\\1\\11
\end{tabular}\endgroup%
\kern3pt%
\begingroup \smaller\smaller\smaller\begin{tabular}{@{}c@{}}%
30\\-1\\11
\end{tabular}\endgroup%
\kern3pt%
\begingroup \smaller\smaller\smaller\begin{tabular}{@{}c@{}}%
90\\-8\\30
\end{tabular}\endgroup%
\kern3pt%
\begingroup \smaller\smaller\smaller\begin{tabular}{@{}c@{}}%
90\\-11\\27
\end{tabular}\endgroup%
\kern3pt%
\begingroup \smaller\smaller\smaller\begin{tabular}{@{}c@{}}%
5\\-1\\1
\end{tabular}\endgroup%
\kern3pt%
\begingroup \smaller\smaller\smaller\begin{tabular}{@{}c@{}}%
90\\-19\\3
\end{tabular}\endgroup%
{$\left.\llap{\phantom{%
\begingroup \smaller\smaller\smaller\begin{tabular}{@{}c@{}}%
0\\0\\0
\end{tabular}\endgroup%
}}\!\right]$}%
}%
\ifdim\wd\matricesbox>\halfwidth\myboxwidth=\hsize\else\myboxwidth=\halfwidth\fi
\vbox{%
\ifdim\myboxwidth=\hsize
\setbox\onelinebox=\hbox{%
\vbox{\hbox{%
$\Pi_{19,16}$ spans $L_{16.13}$%
}\hbox{%
$3636322|223636322\slashthree22\rtimes D_{2}$%
}%
}%
\hfill\copy\matricesbox
}%
\ifdim\wd\onelinebox>\myboxwidth
\hbox to \myboxwidth{%
$\Pi_{19,16}$ spans $L_{16.13}$%
\hfil
$3636322|223636322\slashthree22\rtimes D_{2}$%
}%
\box\matricesbox
\else
\hbox to \myboxwidth{%
\unhbox\onelinebox
}%
\fi
\else
\hbox to \myboxwidth{%
$\Pi_{19,16}$ spans $L_{16.13}$%
\hfil}%
\hbox to \myboxwidth{%
$3636322|223636322\slashthree22\rtimes D_{2}$%
\hfil}%
\box\matricesbox
\fi
}%
\hfill\discretionary{}{}{}%
\setbox\matricesbox=\hbox{%
{$\left[\!\llap{\phantom{%
\begingroup \smaller\smaller\smaller\begin{tabular}{@{}c@{}}%
\phantom{0}\\\phantom{0}\\\phantom{0}
\end{tabular}\endgroup%
}}\right.$}%
\begingroup \smaller\smaller\smaller\begin{tabular}{@{}c@{}}%
-1\\\phantom{0}\\\phantom{0}
\end{tabular}\endgroup%
\kern3pt%
\begingroup \smaller\smaller\smaller\begin{tabular}{@{}c@{}}%
\phantom{0}\\15/2\\\phantom{0}
\end{tabular}\endgroup%
\kern3pt%
\begingroup \smaller\smaller\smaller\begin{tabular}{@{}c@{}}%
\phantom{0}\\\phantom{0}\\45/2
\end{tabular}\endgroup%
{$\left.\llap{\phantom{%
\begingroup \smaller\smaller\smaller\begin{tabular}{@{}c@{}}%
\phantom{0}\\\phantom{0}\\\phantom{0}
\end{tabular}\endgroup%
}}\!\right]$}%
{$\left[\!\llap{\phantom{%
\begingroup \smaller\smaller\smaller\begin{tabular}{@{}c@{}}%
0\\0\\0
\end{tabular}\endgroup%
}}\right.$}%
\begingroup \smaller\smaller\smaller\begin{tabular}{@{}c@{}}%
30\\11\\1
\end{tabular}\endgroup%
\kern3pt%
\begingroup \smaller\smaller\smaller\begin{tabular}{@{}c@{}}%
9\\3\\1
\end{tabular}\endgroup%
\kern3pt%
\begingroup \smaller\smaller\smaller\begin{tabular}{@{}c@{}}%
5\\1\\1
\end{tabular}\endgroup%
\kern3pt%
\begingroup \smaller\smaller\smaller\begin{tabular}{@{}c@{}}%
90\\3\\19
\end{tabular}\endgroup%
\kern3pt%
\begingroup \smaller\smaller\smaller\begin{tabular}{@{}c@{}}%
90\\-3\\19
\end{tabular}\endgroup%
\kern3pt%
\begingroup \smaller\smaller\smaller\begin{tabular}{@{}c@{}}%
30\\-4\\6
\end{tabular}\endgroup%
\kern3pt%
\begingroup \smaller\smaller\smaller\begin{tabular}{@{}c@{}}%
30\\-7\\5
\end{tabular}\endgroup%
\kern3pt%
\begingroup \smaller\smaller\smaller\begin{tabular}{@{}c@{}}%
90\\-27\\11
\end{tabular}\endgroup%
\kern3pt%
\begingroup \smaller\smaller\smaller\begin{tabular}{@{}c@{}}%
90\\-30\\8
\end{tabular}\endgroup%
\kern3pt%
\begingroup \smaller\smaller\smaller\begin{tabular}{@{}c@{}}%
5\\-2\\0
\end{tabular}\endgroup%
{$\left.\llap{\phantom{%
\begingroup \smaller\smaller\smaller\begin{tabular}{@{}c@{}}%
0\\0\\0
\end{tabular}\endgroup%
}}\!\right]$}%
}%
\ifdim\wd\matricesbox>\halfwidth\myboxwidth=\hsize\else\myboxwidth=\halfwidth\fi
\vbox{%
\ifdim\myboxwidth=\hsize
\setbox\onelinebox=\hbox{%
\vbox{\hbox{%
$\Pi_{19,17}$ spans $L_{16.13}$%
}\hbox{%
$36363222\slashthree222363632|2\rtimes D_{2}$%
}%
}%
\hfill\copy\matricesbox
}%
\ifdim\wd\onelinebox>\myboxwidth
\hbox to \myboxwidth{%
$\Pi_{19,17}$ spans $L_{16.13}$%
\hfil
$36363222\slashthree222363632|2\rtimes D_{2}$%
}%
\box\matricesbox
\else
\hbox to \myboxwidth{%
\unhbox\onelinebox
}%
\fi
\else
\hbox to \myboxwidth{%
$\Pi_{19,17}$ spans $L_{16.13}$%
\hfil}%
\hbox to \myboxwidth{%
$36363222\slashthree222363632|2\rtimes D_{2}$%
\hfil}%
\box\matricesbox
\fi
}%
\hfill\discretionary{}{}{}%
\setbox\matricesbox=\hbox{%
{$\left[\!\llap{\phantom{%
\begingroup \smaller\smaller\smaller\begin{tabular}{@{}c@{}}%
\phantom{0}\\\phantom{0}\\\phantom{0}
\end{tabular}\endgroup%
}}\right.$}%
\begingroup \smaller\smaller\smaller\begin{tabular}{@{}c@{}}%
-1\\\phantom{0}\\\phantom{0}
\end{tabular}\endgroup%
\kern3pt%
\begingroup \smaller\smaller\smaller\begin{tabular}{@{}c@{}}%
\phantom{0}\\45/2\\\phantom{0}
\end{tabular}\endgroup%
\kern3pt%
\begingroup \smaller\smaller\smaller\begin{tabular}{@{}c@{}}%
\phantom{0}\\\phantom{0}\\15/2
\end{tabular}\endgroup%
{$\left.\llap{\phantom{%
\begingroup \smaller\smaller\smaller\begin{tabular}{@{}c@{}}%
\phantom{0}\\\phantom{0}\\\phantom{0}
\end{tabular}\endgroup%
}}\!\right]$}%
{$\left[\!\llap{\phantom{%
\begingroup \smaller\smaller\smaller\begin{tabular}{@{}c@{}}%
0\\0\\0
\end{tabular}\endgroup%
}}\right.$}%
\begingroup \smaller\smaller\smaller\begin{tabular}{@{}c@{}}%
90\\19\\3
\end{tabular}\endgroup%
\kern3pt%
\begingroup \smaller\smaller\smaller\begin{tabular}{@{}c@{}}%
30\\6\\4
\end{tabular}\endgroup%
\kern3pt%
\begingroup \smaller\smaller\smaller\begin{tabular}{@{}c@{}}%
30\\5\\7
\end{tabular}\endgroup%
\kern3pt%
\begingroup \smaller\smaller\smaller\begin{tabular}{@{}c@{}}%
90\\11\\27
\end{tabular}\endgroup%
\kern3pt%
\begingroup \smaller\smaller\smaller\begin{tabular}{@{}c@{}}%
90\\8\\30
\end{tabular}\endgroup%
\kern3pt%
\begingroup \smaller\smaller\smaller\begin{tabular}{@{}c@{}}%
5\\0\\2
\end{tabular}\endgroup%
\kern3pt%
\begingroup \smaller\smaller\smaller\begin{tabular}{@{}c@{}}%
9\\-1\\3
\end{tabular}\endgroup%
\kern3pt%
\begingroup \smaller\smaller\smaller\begin{tabular}{@{}c@{}}%
30\\-5\\7
\end{tabular}\endgroup%
\kern3pt%
\begingroup \smaller\smaller\smaller\begin{tabular}{@{}c@{}}%
30\\-6\\4
\end{tabular}\endgroup%
\kern3pt%
\begingroup \smaller\smaller\smaller\begin{tabular}{@{}c@{}}%
9\\-2\\0
\end{tabular}\endgroup%
{$\left.\llap{\phantom{%
\begingroup \smaller\smaller\smaller\begin{tabular}{@{}c@{}}%
0\\0\\0
\end{tabular}\endgroup%
}}\!\right]$}%
}%
\ifdim\wd\matricesbox>\halfwidth\myboxwidth=\hsize\else\myboxwidth=\halfwidth\fi
\vbox{%
\ifdim\myboxwidth=\hsize
\setbox\onelinebox=\hbox{%
\vbox{\hbox{%
$\Pi_{19,18}$ spans $L_{16.13}$%
}\hbox{%
$\slashthree636322232|232223636\rtimes D_{2}$%
}%
}%
\hfill\copy\matricesbox
}%
\ifdim\wd\onelinebox>\myboxwidth
\hbox to \myboxwidth{%
$\Pi_{19,18}$ spans $L_{16.13}$%
\hfil
$\slashthree636322232|232223636\rtimes D_{2}$%
}%
\box\matricesbox
\else
\hbox to \myboxwidth{%
\unhbox\onelinebox
}%
\fi
\else
\hbox to \myboxwidth{%
$\Pi_{19,18}$ spans $L_{16.13}$%
\hfil}%
\hbox to \myboxwidth{%
$\slashthree636322232|232223636\rtimes D_{2}$%
\hfil}%
\box\matricesbox
\fi
}%
\hfill\discretionary{}{}{}%
\setbox\matricesbox=\hbox{%
{$\left[\!\llap{\phantom{%
\begingroup \smaller\smaller\smaller\begin{tabular}{@{}c@{}}%
\phantom{0}\\\phantom{0}\\\phantom{0}
\end{tabular}\endgroup%
}}\right.$}%
\begingroup \smaller\smaller\smaller\begin{tabular}{@{}c@{}}%
-1\\\phantom{0}\\\phantom{0}
\end{tabular}\endgroup%
\kern3pt%
\begingroup \smaller\smaller\smaller\begin{tabular}{@{}c@{}}%
\phantom{0}\\45/2\\\phantom{0}
\end{tabular}\endgroup%
\kern3pt%
\begingroup \smaller\smaller\smaller\begin{tabular}{@{}c@{}}%
\phantom{0}\\\phantom{0}\\15/2
\end{tabular}\endgroup%
{$\left.\llap{\phantom{%
\begingroup \smaller\smaller\smaller\begin{tabular}{@{}c@{}}%
\phantom{0}\\\phantom{0}\\\phantom{0}
\end{tabular}\endgroup%
}}\!\right]$}%
{$\left[\!\llap{\phantom{%
\begingroup \smaller\smaller\smaller\begin{tabular}{@{}c@{}}%
0\\0\\0
\end{tabular}\endgroup%
}}\right.$}%
\begingroup \smaller\smaller\smaller\begin{tabular}{@{}c@{}}%
90\\19\\3
\end{tabular}\endgroup%
\kern3pt%
\begingroup \smaller\smaller\smaller\begin{tabular}{@{}c@{}}%
30\\6\\4
\end{tabular}\endgroup%
\kern3pt%
\begingroup \smaller\smaller\smaller\begin{tabular}{@{}c@{}}%
30\\5\\7
\end{tabular}\endgroup%
\kern3pt%
\begingroup \smaller\smaller\smaller\begin{tabular}{@{}c@{}}%
90\\11\\27
\end{tabular}\endgroup%
\kern3pt%
\begingroup \smaller\smaller\smaller\begin{tabular}{@{}c@{}}%
90\\8\\30
\end{tabular}\endgroup%
\kern3pt%
\begingroup \smaller\smaller\smaller\begin{tabular}{@{}c@{}}%
5\\0\\2
\end{tabular}\endgroup%
\kern3pt%
\begingroup \smaller\smaller\smaller\begin{tabular}{@{}c@{}}%
90\\-8\\30
\end{tabular}\endgroup%
\kern3pt%
\begingroup \smaller\smaller\smaller\begin{tabular}{@{}c@{}}%
90\\-11\\27
\end{tabular}\endgroup%
\kern3pt%
\begingroup \smaller\smaller\smaller\begin{tabular}{@{}c@{}}%
5\\-1\\1
\end{tabular}\endgroup%
\kern3pt%
\begingroup \smaller\smaller\smaller\begin{tabular}{@{}c@{}}%
9\\-2\\0
\end{tabular}\endgroup%
{$\left.\llap{\phantom{%
\begingroup \smaller\smaller\smaller\begin{tabular}{@{}c@{}}%
0\\0\\0
\end{tabular}\endgroup%
}}\!\right]$}%
}%
\ifdim\wd\matricesbox>\halfwidth\myboxwidth=\hsize\else\myboxwidth=\halfwidth\fi
\vbox{%
\ifdim\myboxwidth=\hsize
\setbox\onelinebox=\hbox{%
\vbox{\hbox{%
$\Pi_{19,19}$ spans $L_{16.13}$%
}\hbox{%
$\slashthree636322322|223223636\rtimes D_{2}$%
}%
}%
\hfill\copy\matricesbox
}%
\ifdim\wd\onelinebox>\myboxwidth
\hbox to \myboxwidth{%
$\Pi_{19,19}$ spans $L_{16.13}$%
\hfil
$\slashthree636322322|223223636\rtimes D_{2}$%
}%
\box\matricesbox
\else
\hbox to \myboxwidth{%
\unhbox\onelinebox
}%
\fi
\else
\hbox to \myboxwidth{%
$\Pi_{19,19}$ spans $L_{16.13}$%
\hfil}%
\hbox to \myboxwidth{%
$\slashthree636322322|223223636\rtimes D_{2}$%
\hfil}%
\box\matricesbox
\fi
}%
\hfill\discretionary{}{}{}%
\setbox\matricesbox=\hbox{%
{$\left[\!\llap{\phantom{%
\begingroup \smaller\smaller\smaller\begin{tabular}{@{}c@{}}%
\phantom{0}\\\phantom{0}\\\phantom{0}
\end{tabular}\endgroup%
}}\right.$}%
\begingroup \smaller\smaller\smaller\begin{tabular}{@{}c@{}}%
-5\\\phantom{0}\\\phantom{0}
\end{tabular}\endgroup%
\kern3pt%
\begingroup \smaller\smaller\smaller\begin{tabular}{@{}c@{}}%
\phantom{0}\\3/2\\\phantom{0}
\end{tabular}\endgroup%
\kern3pt%
\begingroup \smaller\smaller\smaller\begin{tabular}{@{}c@{}}%
\phantom{0}\\\phantom{0}\\9/2
\end{tabular}\endgroup%
{$\left.\llap{\phantom{%
\begingroup \smaller\smaller\smaller\begin{tabular}{@{}c@{}}%
\phantom{0}\\\phantom{0}\\\phantom{0}
\end{tabular}\endgroup%
}}\!\right]$}%
{$\left[\!\llap{\phantom{%
\begingroup \smaller\smaller\smaller\begin{tabular}{@{}c@{}}%
0\\0\\0
\end{tabular}\endgroup%
}}\right.$}%
\begingroup \smaller\smaller\smaller\begin{tabular}{@{}c@{}}%
6\\11\\-1
\end{tabular}\endgroup%
\kern3pt%
\begingroup \smaller\smaller\smaller\begin{tabular}{@{}c@{}}%
18\\30\\-8
\end{tabular}\endgroup%
\kern3pt%
\begingroup \smaller\smaller\smaller\begin{tabular}{@{}c@{}}%
18\\27\\-11
\end{tabular}\endgroup%
\kern3pt%
\begingroup \smaller\smaller\smaller\begin{tabular}{@{}c@{}}%
1\\1\\-1
\end{tabular}\endgroup%
\kern3pt%
\begingroup \smaller\smaller\smaller\begin{tabular}{@{}c@{}}%
18\\3\\-19
\end{tabular}\endgroup%
\kern3pt%
\begingroup \smaller\smaller\smaller\begin{tabular}{@{}c@{}}%
18\\-3\\-19
\end{tabular}\endgroup%
\kern3pt%
\begingroup \smaller\smaller\smaller\begin{tabular}{@{}c@{}}%
1\\-1\\-1
\end{tabular}\endgroup%
\kern3pt%
\begingroup \smaller\smaller\smaller\begin{tabular}{@{}c@{}}%
18\\-27\\-11
\end{tabular}\endgroup%
\kern3pt%
\begingroup \smaller\smaller\smaller\begin{tabular}{@{}c@{}}%
18\\-30\\-8
\end{tabular}\endgroup%
\kern3pt%
\begingroup \smaller\smaller\smaller\begin{tabular}{@{}c@{}}%
1\\-2\\0
\end{tabular}\endgroup%
{$\left.\llap{\phantom{%
\begingroup \smaller\smaller\smaller\begin{tabular}{@{}c@{}}%
0\\0\\0
\end{tabular}\endgroup%
}}\!\right]$}%
}%
\ifdim\wd\matricesbox>\halfwidth\myboxwidth=\hsize\else\myboxwidth=\halfwidth\fi
\vbox{%
\ifdim\myboxwidth=\hsize
\setbox\onelinebox=\hbox{%
\vbox{\hbox{%
$\Pi_{19,20}$ spans $L_{16.7}$%
}\hbox{%
$36\slashthree632232232|2322322\rtimes D_{2}$%
}%
}%
\hfill\copy\matricesbox
}%
\ifdim\wd\onelinebox>\myboxwidth
\hbox to \myboxwidth{%
$\Pi_{19,20}$ spans $L_{16.7}$%
\hfil
$36\slashthree632232232|2322322\rtimes D_{2}$%
}%
\box\matricesbox
\else
\hbox to \myboxwidth{%
\unhbox\onelinebox
}%
\fi
\else
\hbox to \myboxwidth{%
$\Pi_{19,20}$ spans $L_{16.7}$%
\hfil}%
\hbox to \myboxwidth{%
$36\slashthree632232232|2322322\rtimes D_{2}$%
\hfil}%
\box\matricesbox
\fi
}%
\hfill\discretionary{}{}{}%
\setbox\matricesbox=\hbox{%
{$\left[\!\llap{\phantom{%
\begingroup \smaller\smaller\smaller
\endgroup%
}}\!\right]$}%
}%
\ifdim\wd\matricesbox>\halfwidth\myboxwidth=\hsize\else\myboxwidth=\halfwidth\fi
\vbox{%
\ifdim\myboxwidth=\hsize
\setbox\onelinebox=\hbox{%
\vbox{\hbox{%
$\Pi_{19,21}$ spans $L_{16.13}$%
}\hbox{%
$3632222232236363636$%
}%
}%
\hfill\copy\matricesbox
}%
\ifdim\wd\onelinebox>\myboxwidth
\hbox to \myboxwidth{%
$\Pi_{19,21}$ spans $L_{16.13}$%
\hfil
$3632222232236363636$%
}%
\box\matricesbox
\else
\hbox to \myboxwidth{%
\unhbox\onelinebox
}%
\fi
\else
\hbox to \myboxwidth{%
$\Pi_{19,21}$ spans $L_{16.13}$%
\hfil}%
\hbox to \myboxwidth{%
$3632222232236363636$%
\hfil}%
\box\matricesbox
\fi
}%
\hfill\discretionary{}{}{}%
\setbox\matricesbox=\hbox{%
{$\left[\!\llap{\phantom{%
\begingroup \smaller\smaller\smaller
\endgroup%
}}\!\right]$}%
}%
\ifdim\wd\matricesbox>\halfwidth\myboxwidth=\hsize\else\myboxwidth=\halfwidth\fi
\vbox{%
\ifdim\myboxwidth=\hsize
\setbox\onelinebox=\hbox{%
\vbox{\hbox{%
$\Pi_{19,22}$ spans $L_{16.13}$%
}\hbox{%
$3632222236322363636$%
}%
}%
\hfill\copy\matricesbox
}%
\ifdim\wd\onelinebox>\myboxwidth
\hbox to \myboxwidth{%
$\Pi_{19,22}$ spans $L_{16.13}$%
\hfil
$3632222236322363636$%
}%
\box\matricesbox
\else
\hbox to \myboxwidth{%
\unhbox\onelinebox
}%
\fi
\else
\hbox to \myboxwidth{%
$\Pi_{19,22}$ spans $L_{16.13}$%
\hfil}%
\hbox to \myboxwidth{%
$3632222236322363636$%
\hfil}%
\box\matricesbox
\fi
}%
\hfill\discretionary{}{}{}%
\setbox\matricesbox=\hbox{%
{$\left[\!\llap{\phantom{%
\begingroup \smaller\smaller\smaller
\endgroup%
}}\!\right]$}%
}%
\ifdim\wd\matricesbox>\halfwidth\myboxwidth=\hsize\else\myboxwidth=\halfwidth\fi
\vbox{%
\ifdim\myboxwidth=\hsize
\setbox\onelinebox=\hbox{%
\vbox{\hbox{%
$\Pi_{19,23}$ spans $L_{16.13}$%
}\hbox{%
$3632222236363223636$%
}%
}%
\hfill\copy\matricesbox
}%
\ifdim\wd\onelinebox>\myboxwidth
\hbox to \myboxwidth{%
$\Pi_{19,23}$ spans $L_{16.13}$%
\hfil
$3632222236363223636$%
}%
\box\matricesbox
\else
\hbox to \myboxwidth{%
\unhbox\onelinebox
}%
\fi
\else
\hbox to \myboxwidth{%
$\Pi_{19,23}$ spans $L_{16.13}$%
\hfil}%
\hbox to \myboxwidth{%
$3632222236363223636$%
\hfil}%
\box\matricesbox
\fi
}%
\hfill\discretionary{}{}{}%
\setbox\matricesbox=\hbox{%
{$\left[\!\llap{\phantom{%
\begingroup \smaller\smaller\smaller
\endgroup%
}}\!\right]$}%
}%
\ifdim\wd\matricesbox>\halfwidth\myboxwidth=\hsize\else\myboxwidth=\halfwidth\fi
\vbox{%
\ifdim\myboxwidth=\hsize
\setbox\onelinebox=\hbox{%
\vbox{\hbox{%
$\Pi_{19,24}$ spans $L_{16.13}$%
}\hbox{%
$3632222236363632236$%
}%
}%
\hfill\copy\matricesbox
}%
\ifdim\wd\onelinebox>\myboxwidth
\hbox to \myboxwidth{%
$\Pi_{19,24}$ spans $L_{16.13}$%
\hfil
$3632222236363632236$%
}%
\box\matricesbox
\else
\hbox to \myboxwidth{%
\unhbox\onelinebox
}%
\fi
\else
\hbox to \myboxwidth{%
$\Pi_{19,24}$ spans $L_{16.13}$%
\hfil}%
\hbox to \myboxwidth{%
$3632222236363632236$%
\hfil}%
\box\matricesbox
\fi
}%
\hfill\discretionary{}{}{}%
\setbox\matricesbox=\hbox{%
{$\left[\!\llap{\phantom{%
\begingroup \smaller\smaller\smaller
\endgroup%
}}\!\right]$}%
}%
\ifdim\wd\matricesbox>\halfwidth\myboxwidth=\hsize\else\myboxwidth=\halfwidth\fi
\vbox{%
\ifdim\myboxwidth=\hsize
\setbox\onelinebox=\hbox{%
\vbox{\hbox{%
$\Pi_{19,25}$ spans $L_{16.13}$%
}\hbox{%
$3632222236363636322$%
}%
}%
\hfill\copy\matricesbox
}%
\ifdim\wd\onelinebox>\myboxwidth
\hbox to \myboxwidth{%
$\Pi_{19,25}$ spans $L_{16.13}$%
\hfil
$3632222236363636322$%
}%
\box\matricesbox
\else
\hbox to \myboxwidth{%
\unhbox\onelinebox
}%
\fi
\else
\hbox to \myboxwidth{%
$\Pi_{19,25}$ spans $L_{16.13}$%
\hfil}%
\hbox to \myboxwidth{%
$3632222236363636322$%
\hfil}%
\box\matricesbox
\fi
}%
\hfill\discretionary{}{}{}%
\setbox\matricesbox=\hbox{%
{$\left[\!\llap{\phantom{%
\begingroup \smaller\smaller\smaller
\endgroup%
}}\!\right]$}%
}%
\ifdim\wd\matricesbox>\halfwidth\myboxwidth=\hsize\else\myboxwidth=\halfwidth\fi
\vbox{%
\ifdim\myboxwidth=\hsize
\setbox\onelinebox=\hbox{%
\vbox{\hbox{%
$\Pi_{19,26}$ spans $L_{16.13}$%
}\hbox{%
$3632222322236363636$%
}%
}%
\hfill\copy\matricesbox
}%
\ifdim\wd\onelinebox>\myboxwidth
\hbox to \myboxwidth{%
$\Pi_{19,26}$ spans $L_{16.13}$%
\hfil
$3632222322236363636$%
}%
\box\matricesbox
\else
\hbox to \myboxwidth{%
\unhbox\onelinebox
}%
\fi
\else
\hbox to \myboxwidth{%
$\Pi_{19,26}$ spans $L_{16.13}$%
\hfil}%
\hbox to \myboxwidth{%
$3632222322236363636$%
\hfil}%
\box\matricesbox
\fi
}%
\hfill\discretionary{}{}{}%
\setbox\matricesbox=\hbox{%
{$\left[\!\llap{\phantom{%
\begingroup \smaller\smaller\smaller
\endgroup%
}}\!\right]$}%
}%
\ifdim\wd\matricesbox>\halfwidth\myboxwidth=\hsize\else\myboxwidth=\halfwidth\fi
\vbox{%
\ifdim\myboxwidth=\hsize
\setbox\onelinebox=\hbox{%
\vbox{\hbox{%
$\Pi_{19,27}$ spans $L_{16.13}$%
}\hbox{%
$3632222322322363636$%
}%
}%
\hfill\copy\matricesbox
}%
\ifdim\wd\onelinebox>\myboxwidth
\hbox to \myboxwidth{%
$\Pi_{19,27}$ spans $L_{16.13}$%
\hfil
$3632222322322363636$%
}%
\box\matricesbox
\else
\hbox to \myboxwidth{%
\unhbox\onelinebox
}%
\fi
\else
\hbox to \myboxwidth{%
$\Pi_{19,27}$ spans $L_{16.13}$%
\hfil}%
\hbox to \myboxwidth{%
$3632222322322363636$%
\hfil}%
\box\matricesbox
\fi
}%
\hfill\discretionary{}{}{}%
\setbox\matricesbox=\hbox{%
{$\left[\!\llap{\phantom{%
\begingroup \smaller\smaller\smaller
\endgroup%
}}\!\right]$}%
}%
\ifdim\wd\matricesbox>\halfwidth\myboxwidth=\hsize\else\myboxwidth=\halfwidth\fi
\vbox{%
\ifdim\myboxwidth=\hsize
\setbox\onelinebox=\hbox{%
\vbox{\hbox{%
$\Pi_{19,28}$ spans $L_{16.13}$%
}\hbox{%
$3632222322363223636$%
}%
}%
\hfill\copy\matricesbox
}%
\ifdim\wd\onelinebox>\myboxwidth
\hbox to \myboxwidth{%
$\Pi_{19,28}$ spans $L_{16.13}$%
\hfil
$3632222322363223636$%
}%
\box\matricesbox
\else
\hbox to \myboxwidth{%
\unhbox\onelinebox
}%
\fi
\else
\hbox to \myboxwidth{%
$\Pi_{19,28}$ spans $L_{16.13}$%
\hfil}%
\hbox to \myboxwidth{%
$3632222322363223636$%
\hfil}%
\box\matricesbox
\fi
}%
\hfill\discretionary{}{}{}%
\setbox\matricesbox=\hbox{%
{$\left[\!\llap{\phantom{%
\begingroup \smaller\smaller\smaller
\endgroup%
}}\!\right]$}%
}%
\ifdim\wd\matricesbox>\halfwidth\myboxwidth=\hsize\else\myboxwidth=\halfwidth\fi
\vbox{%
\ifdim\myboxwidth=\hsize
\setbox\onelinebox=\hbox{%
\vbox{\hbox{%
$\Pi_{19,29}$ spans $L_{16.13}$%
}\hbox{%
$3632222322363632236$%
}%
}%
\hfill\copy\matricesbox
}%
\ifdim\wd\onelinebox>\myboxwidth
\hbox to \myboxwidth{%
$\Pi_{19,29}$ spans $L_{16.13}$%
\hfil
$3632222322363632236$%
}%
\box\matricesbox
\else
\hbox to \myboxwidth{%
\unhbox\onelinebox
}%
\fi
\else
\hbox to \myboxwidth{%
$\Pi_{19,29}$ spans $L_{16.13}$%
\hfil}%
\hbox to \myboxwidth{%
$3632222322363632236$%
\hfil}%
\box\matricesbox
\fi
}%
\hfill\discretionary{}{}{}%
\setbox\matricesbox=\hbox{%
{$\left[\!\llap{\phantom{%
\begingroup \smaller\smaller\smaller
\endgroup%
}}\!\right]$}%
}%
\ifdim\wd\matricesbox>\halfwidth\myboxwidth=\hsize\else\myboxwidth=\halfwidth\fi
\vbox{%
\ifdim\myboxwidth=\hsize
\setbox\onelinebox=\hbox{%
\vbox{\hbox{%
$\Pi_{19,30}$ spans $L_{16.13}$%
}\hbox{%
$3632222322363636322$%
}%
}%
\hfill\copy\matricesbox
}%
\ifdim\wd\onelinebox>\myboxwidth
\hbox to \myboxwidth{%
$\Pi_{19,30}$ spans $L_{16.13}$%
\hfil
$3632222322363636322$%
}%
\box\matricesbox
\else
\hbox to \myboxwidth{%
\unhbox\onelinebox
}%
\fi
\else
\hbox to \myboxwidth{%
$\Pi_{19,30}$ spans $L_{16.13}$%
\hfil}%
\hbox to \myboxwidth{%
$3632222322363636322$%
\hfil}%
\box\matricesbox
\fi
}%
\hfill\discretionary{}{}{}%
\setbox\matricesbox=\hbox{%
{$\left[\!\llap{\phantom{%
\begingroup \smaller\smaller\smaller
\endgroup%
}}\!\right]$}%
}%
\ifdim\wd\matricesbox>\halfwidth\myboxwidth=\hsize\else\myboxwidth=\halfwidth\fi
\vbox{%
\ifdim\myboxwidth=\hsize
\setbox\onelinebox=\hbox{%
\vbox{\hbox{%
$\Pi_{19,31}$ spans $L_{16.13}$%
}\hbox{%
$3632222363222363636$%
}%
}%
\hfill\copy\matricesbox
}%
\ifdim\wd\onelinebox>\myboxwidth
\hbox to \myboxwidth{%
$\Pi_{19,31}$ spans $L_{16.13}$%
\hfil
$3632222363222363636$%
}%
\box\matricesbox
\else
\hbox to \myboxwidth{%
\unhbox\onelinebox
}%
\fi
\else
\hbox to \myboxwidth{%
$\Pi_{19,31}$ spans $L_{16.13}$%
\hfil}%
\hbox to \myboxwidth{%
$3632222363222363636$%
\hfil}%
\box\matricesbox
\fi
}%
\hfill\discretionary{}{}{}%
\setbox\matricesbox=\hbox{%
{$\left[\!\llap{\phantom{%
\begingroup \smaller\smaller\smaller
\endgroup%
}}\!\right]$}%
}%
\ifdim\wd\matricesbox>\halfwidth\myboxwidth=\hsize\else\myboxwidth=\halfwidth\fi
\vbox{%
\ifdim\myboxwidth=\hsize
\setbox\onelinebox=\hbox{%
\vbox{\hbox{%
$\Pi_{19,32}$ spans $L_{16.13}$%
}\hbox{%
$3632222363223223636$%
}%
}%
\hfill\copy\matricesbox
}%
\ifdim\wd\onelinebox>\myboxwidth
\hbox to \myboxwidth{%
$\Pi_{19,32}$ spans $L_{16.13}$%
\hfil
$3632222363223223636$%
}%
\box\matricesbox
\else
\hbox to \myboxwidth{%
\unhbox\onelinebox
}%
\fi
\else
\hbox to \myboxwidth{%
$\Pi_{19,32}$ spans $L_{16.13}$%
\hfil}%
\hbox to \myboxwidth{%
$3632222363223223636$%
\hfil}%
\box\matricesbox
\fi
}%
\hfill\discretionary{}{}{}%
\setbox\matricesbox=\hbox{%
{$\left[\!\llap{\phantom{%
\begingroup \smaller\smaller\smaller
\endgroup%
}}\!\right]$}%
}%
\ifdim\wd\matricesbox>\halfwidth\myboxwidth=\hsize\else\myboxwidth=\halfwidth\fi
\vbox{%
\ifdim\myboxwidth=\hsize
\setbox\onelinebox=\hbox{%
\vbox{\hbox{%
$\Pi_{19,33}$ spans $L_{16.13}$%
}\hbox{%
$3632222363223632236$%
}%
}%
\hfill\copy\matricesbox
}%
\ifdim\wd\onelinebox>\myboxwidth
\hbox to \myboxwidth{%
$\Pi_{19,33}$ spans $L_{16.13}$%
\hfil
$3632222363223632236$%
}%
\box\matricesbox
\else
\hbox to \myboxwidth{%
\unhbox\onelinebox
}%
\fi
\else
\hbox to \myboxwidth{%
$\Pi_{19,33}$ spans $L_{16.13}$%
\hfil}%
\hbox to \myboxwidth{%
$3632222363223632236$%
\hfil}%
\box\matricesbox
\fi
}%
\hfill\discretionary{}{}{}%
\setbox\matricesbox=\hbox{%
{$\left[\!\llap{\phantom{%
\begingroup \smaller\smaller\smaller
\endgroup%
}}\!\right]$}%
}%
\ifdim\wd\matricesbox>\halfwidth\myboxwidth=\hsize\else\myboxwidth=\halfwidth\fi
\vbox{%
\ifdim\myboxwidth=\hsize
\setbox\onelinebox=\hbox{%
\vbox{\hbox{%
$\Pi_{19,34}$ spans $L_{16.13}$%
}\hbox{%
$3632222363632223636$%
}%
}%
\hfill\copy\matricesbox
}%
\ifdim\wd\onelinebox>\myboxwidth
\hbox to \myboxwidth{%
$\Pi_{19,34}$ spans $L_{16.13}$%
\hfil
$3632222363632223636$%
}%
\box\matricesbox
\else
\hbox to \myboxwidth{%
\unhbox\onelinebox
}%
\fi
\else
\hbox to \myboxwidth{%
$\Pi_{19,34}$ spans $L_{16.13}$%
\hfil}%
\hbox to \myboxwidth{%
$3632222363632223636$%
\hfil}%
\box\matricesbox
\fi
}%
\hfill\discretionary{}{}{}%
\setbox\matricesbox=\hbox{%
{$\left[\!\llap{\phantom{%
\begingroup \smaller\smaller\smaller
\endgroup%
}}\!\right]$}%
}%
\ifdim\wd\matricesbox>\halfwidth\myboxwidth=\hsize\else\myboxwidth=\halfwidth\fi
\vbox{%
\ifdim\myboxwidth=\hsize
\setbox\onelinebox=\hbox{%
\vbox{\hbox{%
$\Pi_{19,35}$ spans $L_{16.13}$%
}\hbox{%
$3632223222236363636$%
}%
}%
\hfill\copy\matricesbox
}%
\ifdim\wd\onelinebox>\myboxwidth
\hbox to \myboxwidth{%
$\Pi_{19,35}$ spans $L_{16.13}$%
\hfil
$3632223222236363636$%
}%
\box\matricesbox
\else
\hbox to \myboxwidth{%
\unhbox\onelinebox
}%
\fi
\else
\hbox to \myboxwidth{%
$\Pi_{19,35}$ spans $L_{16.13}$%
\hfil}%
\hbox to \myboxwidth{%
$3632223222236363636$%
\hfil}%
\box\matricesbox
\fi
}%
\hfill\discretionary{}{}{}%
\setbox\matricesbox=\hbox{%
{$\left[\!\llap{\phantom{%
\begingroup \smaller\smaller\smaller
\endgroup%
}}\!\right]$}%
}%
\ifdim\wd\matricesbox>\halfwidth\myboxwidth=\hsize\else\myboxwidth=\halfwidth\fi
\vbox{%
\ifdim\myboxwidth=\hsize
\setbox\onelinebox=\hbox{%
\vbox{\hbox{%
$\Pi_{19,36}$ spans $L_{16.13}$%
}\hbox{%
$3632223222322363636$%
}%
}%
\hfill\copy\matricesbox
}%
\ifdim\wd\onelinebox>\myboxwidth
\hbox to \myboxwidth{%
$\Pi_{19,36}$ spans $L_{16.13}$%
\hfil
$3632223222322363636$%
}%
\box\matricesbox
\else
\hbox to \myboxwidth{%
\unhbox\onelinebox
}%
\fi
\else
\hbox to \myboxwidth{%
$\Pi_{19,36}$ spans $L_{16.13}$%
\hfil}%
\hbox to \myboxwidth{%
$3632223222322363636$%
\hfil}%
\box\matricesbox
\fi
}%
\hfill\discretionary{}{}{}%
\setbox\matricesbox=\hbox{%
{$\left[\!\llap{\phantom{%
\begingroup \smaller\smaller\smaller
\endgroup%
}}\!\right]$}%
}%
\ifdim\wd\matricesbox>\halfwidth\myboxwidth=\hsize\else\myboxwidth=\halfwidth\fi
\vbox{%
\ifdim\myboxwidth=\hsize
\setbox\onelinebox=\hbox{%
\vbox{\hbox{%
$\Pi_{19,37}$ spans $L_{16.13}$%
}\hbox{%
$3632223222363223636$%
}%
}%
\hfill\copy\matricesbox
}%
\ifdim\wd\onelinebox>\myboxwidth
\hbox to \myboxwidth{%
$\Pi_{19,37}$ spans $L_{16.13}$%
\hfil
$3632223222363223636$%
}%
\box\matricesbox
\else
\hbox to \myboxwidth{%
\unhbox\onelinebox
}%
\fi
\else
\hbox to \myboxwidth{%
$\Pi_{19,37}$ spans $L_{16.13}$%
\hfil}%
\hbox to \myboxwidth{%
$3632223222363223636$%
\hfil}%
\box\matricesbox
\fi
}%
\hfill\discretionary{}{}{}%
\setbox\matricesbox=\hbox{%
{$\left[\!\llap{\phantom{%
\begingroup \smaller\smaller\smaller
\endgroup%
}}\!\right]$}%
}%
\ifdim\wd\matricesbox>\halfwidth\myboxwidth=\hsize\else\myboxwidth=\halfwidth\fi
\vbox{%
\ifdim\myboxwidth=\hsize
\setbox\onelinebox=\hbox{%
\vbox{\hbox{%
$\Pi_{19,38}$ spans $L_{16.13}$%
}\hbox{%
$3632223223223223636$%
}%
}%
\hfill\copy\matricesbox
}%
\ifdim\wd\onelinebox>\myboxwidth
\hbox to \myboxwidth{%
$\Pi_{19,38}$ spans $L_{16.13}$%
\hfil
$3632223223223223636$%
}%
\box\matricesbox
\else
\hbox to \myboxwidth{%
\unhbox\onelinebox
}%
\fi
\else
\hbox to \myboxwidth{%
$\Pi_{19,38}$ spans $L_{16.13}$%
\hfil}%
\hbox to \myboxwidth{%
$3632223223223223636$%
\hfil}%
\box\matricesbox
\fi
}%
\hfill\discretionary{}{}{}%
\setbox\matricesbox=\hbox{%
{$\left[\!\llap{\phantom{%
\begingroup \smaller\smaller\smaller
\endgroup%
}}\!\right]$}%
}%
\ifdim\wd\matricesbox>\halfwidth\myboxwidth=\hsize\else\myboxwidth=\halfwidth\fi
\vbox{%
\ifdim\myboxwidth=\hsize
\setbox\onelinebox=\hbox{%
\vbox{\hbox{%
$\Pi_{19,39}$ spans $L_{16.13}$%
}\hbox{%
$3632223223223632236$%
}%
}%
\hfill\copy\matricesbox
}%
\ifdim\wd\onelinebox>\myboxwidth
\hbox to \myboxwidth{%
$\Pi_{19,39}$ spans $L_{16.13}$%
\hfil
$3632223223223632236$%
}%
\box\matricesbox
\else
\hbox to \myboxwidth{%
\unhbox\onelinebox
}%
\fi
\else
\hbox to \myboxwidth{%
$\Pi_{19,39}$ spans $L_{16.13}$%
\hfil}%
\hbox to \myboxwidth{%
$3632223223223632236$%
\hfil}%
\box\matricesbox
\fi
}%
\hfill\discretionary{}{}{}%
\setbox\matricesbox=\hbox{%
{$\left[\!\llap{\phantom{%
\begingroup \smaller\smaller\smaller
\endgroup%
}}\!\right]$}%
}%
\ifdim\wd\matricesbox>\halfwidth\myboxwidth=\hsize\else\myboxwidth=\halfwidth\fi
\vbox{%
\ifdim\myboxwidth=\hsize
\setbox\onelinebox=\hbox{%
\vbox{\hbox{%
$\Pi_{19,40}$ spans $L_{16.13}$%
}\hbox{%
$3632223223223636322$%
}%
}%
\hfill\copy\matricesbox
}%
\ifdim\wd\onelinebox>\myboxwidth
\hbox to \myboxwidth{%
$\Pi_{19,40}$ spans $L_{16.13}$%
\hfil
$3632223223223636322$%
}%
\box\matricesbox
\else
\hbox to \myboxwidth{%
\unhbox\onelinebox
}%
\fi
\else
\hbox to \myboxwidth{%
$\Pi_{19,40}$ spans $L_{16.13}$%
\hfil}%
\hbox to \myboxwidth{%
$3632223223223636322$%
\hfil}%
\box\matricesbox
\fi
}%
\hfill\discretionary{}{}{}%
\setbox\matricesbox=\hbox{%
{$\left[\!\llap{\phantom{%
\begingroup \smaller\smaller\smaller
\endgroup%
}}\!\right]$}%
}%
\ifdim\wd\matricesbox>\halfwidth\myboxwidth=\hsize\else\myboxwidth=\halfwidth\fi
\vbox{%
\ifdim\myboxwidth=\hsize
\setbox\onelinebox=\hbox{%
\vbox{\hbox{%
$\Pi_{19,41}$ spans $L_{16.13}$%
}\hbox{%
$3632223223632223636$%
}%
}%
\hfill\copy\matricesbox
}%
\ifdim\wd\onelinebox>\myboxwidth
\hbox to \myboxwidth{%
$\Pi_{19,41}$ spans $L_{16.13}$%
\hfil
$3632223223632223636$%
}%
\box\matricesbox
\else
\hbox to \myboxwidth{%
\unhbox\onelinebox
}%
\fi
\else
\hbox to \myboxwidth{%
$\Pi_{19,41}$ spans $L_{16.13}$%
\hfil}%
\hbox to \myboxwidth{%
$3632223223632223636$%
\hfil}%
\box\matricesbox
\fi
}%
\hfill\discretionary{}{}{}%
\setbox\matricesbox=\hbox{%
{$\left[\!\llap{\phantom{%
\begingroup \smaller\smaller\smaller
\endgroup%
}}\!\right]$}%
}%
\ifdim\wd\matricesbox>\halfwidth\myboxwidth=\hsize\else\myboxwidth=\halfwidth\fi
\vbox{%
\ifdim\myboxwidth=\hsize
\setbox\onelinebox=\hbox{%
\vbox{\hbox{%
$\Pi_{19,42}$ spans $L_{16.13}$%
}\hbox{%
$3632223223632232236$%
}%
}%
\hfill\copy\matricesbox
}%
\ifdim\wd\onelinebox>\myboxwidth
\hbox to \myboxwidth{%
$\Pi_{19,42}$ spans $L_{16.13}$%
\hfil
$3632223223632232236$%
}%
\box\matricesbox
\else
\hbox to \myboxwidth{%
\unhbox\onelinebox
}%
\fi
\else
\hbox to \myboxwidth{%
$\Pi_{19,42}$ spans $L_{16.13}$%
\hfil}%
\hbox to \myboxwidth{%
$3632223223632232236$%
\hfil}%
\box\matricesbox
\fi
}%
\hfill\discretionary{}{}{}%
\setbox\matricesbox=\hbox{%
{$\left[\!\llap{\phantom{%
\begingroup \smaller\smaller\smaller
\endgroup%
}}\!\right]$}%
}%
\ifdim\wd\matricesbox>\halfwidth\myboxwidth=\hsize\else\myboxwidth=\halfwidth\fi
\vbox{%
\ifdim\myboxwidth=\hsize
\setbox\onelinebox=\hbox{%
\vbox{\hbox{%
$\Pi_{19,43}$ spans $L_{16.13}$%
}\hbox{%
$3632223223632236322$%
}%
}%
\hfill\copy\matricesbox
}%
\ifdim\wd\onelinebox>\myboxwidth
\hbox to \myboxwidth{%
$\Pi_{19,43}$ spans $L_{16.13}$%
\hfil
$3632223223632236322$%
}%
\box\matricesbox
\else
\hbox to \myboxwidth{%
\unhbox\onelinebox
}%
\fi
\else
\hbox to \myboxwidth{%
$\Pi_{19,43}$ spans $L_{16.13}$%
\hfil}%
\hbox to \myboxwidth{%
$3632223223632236322$%
\hfil}%
\box\matricesbox
\fi
}%
\hfill\discretionary{}{}{}%
\setbox\matricesbox=\hbox{%
{$\left[\!\llap{\phantom{%
\begingroup \smaller\smaller\smaller
\endgroup%
}}\!\right]$}%
}%
\ifdim\wd\matricesbox>\halfwidth\myboxwidth=\hsize\else\myboxwidth=\halfwidth\fi
\vbox{%
\ifdim\myboxwidth=\hsize
\setbox\onelinebox=\hbox{%
\vbox{\hbox{%
$\Pi_{19,44}$ spans $L_{16.13}$%
}\hbox{%
$3632223223636322236$%
}%
}%
\hfill\copy\matricesbox
}%
\ifdim\wd\onelinebox>\myboxwidth
\hbox to \myboxwidth{%
$\Pi_{19,44}$ spans $L_{16.13}$%
\hfil
$3632223223636322236$%
}%
\box\matricesbox
\else
\hbox to \myboxwidth{%
\unhbox\onelinebox
}%
\fi
\else
\hbox to \myboxwidth{%
$\Pi_{19,44}$ spans $L_{16.13}$%
\hfil}%
\hbox to \myboxwidth{%
$3632223223636322236$%
\hfil}%
\box\matricesbox
\fi
}%
\hfill\discretionary{}{}{}%
\setbox\matricesbox=\hbox{%
{$\left[\!\llap{\phantom{%
\begingroup \smaller\smaller\smaller
\endgroup%
}}\!\right]$}%
}%
\ifdim\wd\matricesbox>\halfwidth\myboxwidth=\hsize\else\myboxwidth=\halfwidth\fi
\vbox{%
\ifdim\myboxwidth=\hsize
\setbox\onelinebox=\hbox{%
\vbox{\hbox{%
$\Pi_{19,45}$ spans $L_{16.13}$%
}\hbox{%
$3632223223636322322$%
}%
}%
\hfill\copy\matricesbox
}%
\ifdim\wd\onelinebox>\myboxwidth
\hbox to \myboxwidth{%
$\Pi_{19,45}$ spans $L_{16.13}$%
\hfil
$3632223223636322322$%
}%
\box\matricesbox
\else
\hbox to \myboxwidth{%
\unhbox\onelinebox
}%
\fi
\else
\hbox to \myboxwidth{%
$\Pi_{19,45}$ spans $L_{16.13}$%
\hfil}%
\hbox to \myboxwidth{%
$3632223223636322322$%
\hfil}%
\box\matricesbox
\fi
}%
\hfill\discretionary{}{}{}%
\setbox\matricesbox=\hbox{%
{$\left[\!\llap{\phantom{%
\begingroup \smaller\smaller\smaller
\endgroup%
}}\!\right]$}%
}%
\ifdim\wd\matricesbox>\halfwidth\myboxwidth=\hsize\else\myboxwidth=\halfwidth\fi
\vbox{%
\ifdim\myboxwidth=\hsize
\setbox\onelinebox=\hbox{%
\vbox{\hbox{%
$\Pi_{19,46}$ spans $L_{16.13}$%
}\hbox{%
$3632223223636363222$%
}%
}%
\hfill\copy\matricesbox
}%
\ifdim\wd\onelinebox>\myboxwidth
\hbox to \myboxwidth{%
$\Pi_{19,46}$ spans $L_{16.13}$%
\hfil
$3632223223636363222$%
}%
\box\matricesbox
\else
\hbox to \myboxwidth{%
\unhbox\onelinebox
}%
\fi
\else
\hbox to \myboxwidth{%
$\Pi_{19,46}$ spans $L_{16.13}$%
\hfil}%
\hbox to \myboxwidth{%
$3632223223636363222$%
\hfil}%
\box\matricesbox
\fi
}%
\hfill\discretionary{}{}{}%
\setbox\matricesbox=\hbox{%
{$\left[\!\llap{\phantom{%
\begingroup \smaller\smaller\smaller
\endgroup%
}}\!\right]$}%
}%
\ifdim\wd\matricesbox>\halfwidth\myboxwidth=\hsize\else\myboxwidth=\halfwidth\fi
\vbox{%
\ifdim\myboxwidth=\hsize
\setbox\onelinebox=\hbox{%
\vbox{\hbox{%
$\Pi_{19,47}$ spans $L_{16.13}$%
}\hbox{%
$3632223632223632236$%
}%
}%
\hfill\copy\matricesbox
}%
\ifdim\wd\onelinebox>\myboxwidth
\hbox to \myboxwidth{%
$\Pi_{19,47}$ spans $L_{16.13}$%
\hfil
$3632223632223632236$%
}%
\box\matricesbox
\else
\hbox to \myboxwidth{%
\unhbox\onelinebox
}%
\fi
\else
\hbox to \myboxwidth{%
$\Pi_{19,47}$ spans $L_{16.13}$%
\hfil}%
\hbox to \myboxwidth{%
$3632223632223632236$%
\hfil}%
\box\matricesbox
\fi
}%
\hfill\discretionary{}{}{}%
\setbox\matricesbox=\hbox{%
{$\left[\!\llap{\phantom{%
\begingroup \smaller\smaller\smaller
\endgroup%
}}\!\right]$}%
}%
\ifdim\wd\matricesbox>\halfwidth\myboxwidth=\hsize\else\myboxwidth=\halfwidth\fi
\vbox{%
\ifdim\myboxwidth=\hsize
\setbox\onelinebox=\hbox{%
\vbox{\hbox{%
$\Pi_{19,48}$ spans $L_{16.13}$%
}\hbox{%
$3632223632223636322$%
}%
}%
\hfill\copy\matricesbox
}%
\ifdim\wd\onelinebox>\myboxwidth
\hbox to \myboxwidth{%
$\Pi_{19,48}$ spans $L_{16.13}$%
\hfil
$3632223632223636322$%
}%
\box\matricesbox
\else
\hbox to \myboxwidth{%
\unhbox\onelinebox
}%
\fi
\else
\hbox to \myboxwidth{%
$\Pi_{19,48}$ spans $L_{16.13}$%
\hfil}%
\hbox to \myboxwidth{%
$3632223632223636322$%
\hfil}%
\box\matricesbox
\fi
}%
\hfill\discretionary{}{}{}%
\setbox\matricesbox=\hbox{%
{$\left[\!\llap{\phantom{%
\begingroup \smaller\smaller\smaller
\endgroup%
}}\!\right]$}%
}%
\ifdim\wd\matricesbox>\halfwidth\myboxwidth=\hsize\else\myboxwidth=\halfwidth\fi
\vbox{%
\ifdim\myboxwidth=\hsize
\setbox\onelinebox=\hbox{%
\vbox{\hbox{%
$\Pi_{19,49}$ spans $L_{16.13}$%
}\hbox{%
$3632223632232223636$%
}%
}%
\hfill\copy\matricesbox
}%
\ifdim\wd\onelinebox>\myboxwidth
\hbox to \myboxwidth{%
$\Pi_{19,49}$ spans $L_{16.13}$%
\hfil
$3632223632232223636$%
}%
\box\matricesbox
\else
\hbox to \myboxwidth{%
\unhbox\onelinebox
}%
\fi
\else
\hbox to \myboxwidth{%
$\Pi_{19,49}$ spans $L_{16.13}$%
\hfil}%
\hbox to \myboxwidth{%
$3632223632232223636$%
\hfil}%
\box\matricesbox
\fi
}%
\hfill\discretionary{}{}{}%
\setbox\matricesbox=\hbox{%
{$\left[\!\llap{\phantom{%
\begingroup \smaller\smaller\smaller
\endgroup%
}}\!\right]$}%
}%
\ifdim\wd\matricesbox>\halfwidth\myboxwidth=\hsize\else\myboxwidth=\halfwidth\fi
\vbox{%
\ifdim\myboxwidth=\hsize
\setbox\onelinebox=\hbox{%
\vbox{\hbox{%
$\Pi_{19,50}$ spans $L_{16.13}$%
}\hbox{%
$3632223632232232236$%
}%
}%
\hfill\copy\matricesbox
}%
\ifdim\wd\onelinebox>\myboxwidth
\hbox to \myboxwidth{%
$\Pi_{19,50}$ spans $L_{16.13}$%
\hfil
$3632223632232232236$%
}%
\box\matricesbox
\else
\hbox to \myboxwidth{%
\unhbox\onelinebox
}%
\fi
\else
\hbox to \myboxwidth{%
$\Pi_{19,50}$ spans $L_{16.13}$%
\hfil}%
\hbox to \myboxwidth{%
$3632223632232232236$%
\hfil}%
\box\matricesbox
\fi
}%
\hfill\discretionary{}{}{}%
\setbox\matricesbox=\hbox{%
{$\left[\!\llap{\phantom{%
\begingroup \smaller\smaller\smaller
\endgroup%
}}\!\right]$}%
}%
\ifdim\wd\matricesbox>\halfwidth\myboxwidth=\hsize\else\myboxwidth=\halfwidth\fi
\vbox{%
\ifdim\myboxwidth=\hsize
\setbox\onelinebox=\hbox{%
\vbox{\hbox{%
$\Pi_{19,51}$ spans $L_{16.13}$%
}\hbox{%
$3632223632232236322$%
}%
}%
\hfill\copy\matricesbox
}%
\ifdim\wd\onelinebox>\myboxwidth
\hbox to \myboxwidth{%
$\Pi_{19,51}$ spans $L_{16.13}$%
\hfil
$3632223632232236322$%
}%
\box\matricesbox
\else
\hbox to \myboxwidth{%
\unhbox\onelinebox
}%
\fi
\else
\hbox to \myboxwidth{%
$\Pi_{19,51}$ spans $L_{16.13}$%
\hfil}%
\hbox to \myboxwidth{%
$3632223632232236322$%
\hfil}%
\box\matricesbox
\fi
}%
\hfill\discretionary{}{}{}%
\setbox\matricesbox=\hbox{%
{$\left[\!\llap{\phantom{%
\begingroup \smaller\smaller\smaller
\endgroup%
}}\!\right]$}%
}%
\ifdim\wd\matricesbox>\halfwidth\myboxwidth=\hsize\else\myboxwidth=\halfwidth\fi
\vbox{%
\ifdim\myboxwidth=\hsize
\setbox\onelinebox=\hbox{%
\vbox{\hbox{%
$\Pi_{19,52}$ spans $L_{16.13}$%
}\hbox{%
$3632223632236322322$%
}%
}%
\hfill\copy\matricesbox
}%
\ifdim\wd\onelinebox>\myboxwidth
\hbox to \myboxwidth{%
$\Pi_{19,52}$ spans $L_{16.13}$%
\hfil
$3632223632236322322$%
}%
\box\matricesbox
\else
\hbox to \myboxwidth{%
\unhbox\onelinebox
}%
\fi
\else
\hbox to \myboxwidth{%
$\Pi_{19,52}$ spans $L_{16.13}$%
\hfil}%
\hbox to \myboxwidth{%
$3632223632236322322$%
\hfil}%
\box\matricesbox
\fi
}%
\hfill\discretionary{}{}{}%
\setbox\matricesbox=\hbox{%
{$\left[\!\llap{\phantom{%
\begingroup \smaller\smaller\smaller
\endgroup%
}}\!\right]$}%
}%
\ifdim\wd\matricesbox>\halfwidth\myboxwidth=\hsize\else\myboxwidth=\halfwidth\fi
\vbox{%
\ifdim\myboxwidth=\hsize
\setbox\onelinebox=\hbox{%
\vbox{\hbox{%
$\Pi_{19,53}$ spans $L_{16.13}$%
}\hbox{%
$3632223636322223636$%
}%
}%
\hfill\copy\matricesbox
}%
\ifdim\wd\onelinebox>\myboxwidth
\hbox to \myboxwidth{%
$\Pi_{19,53}$ spans $L_{16.13}$%
\hfil
$3632223636322223636$%
}%
\box\matricesbox
\else
\hbox to \myboxwidth{%
\unhbox\onelinebox
}%
\fi
\else
\hbox to \myboxwidth{%
$\Pi_{19,53}$ spans $L_{16.13}$%
\hfil}%
\hbox to \myboxwidth{%
$3632223636322223636$%
\hfil}%
\box\matricesbox
\fi
}%
\hfill\discretionary{}{}{}%
\setbox\matricesbox=\hbox{%
{$\left[\!\llap{\phantom{%
\begingroup \smaller\smaller\smaller
\endgroup%
}}\!\right]$}%
}%
\ifdim\wd\matricesbox>\halfwidth\myboxwidth=\hsize\else\myboxwidth=\halfwidth\fi
\vbox{%
\ifdim\myboxwidth=\hsize
\setbox\onelinebox=\hbox{%
\vbox{\hbox{%
$\Pi_{19,54}$ spans $L_{16.13}$%
}\hbox{%
$3632223636322232236$%
}%
}%
\hfill\copy\matricesbox
}%
\ifdim\wd\onelinebox>\myboxwidth
\hbox to \myboxwidth{%
$\Pi_{19,54}$ spans $L_{16.13}$%
\hfil
$3632223636322232236$%
}%
\box\matricesbox
\else
\hbox to \myboxwidth{%
\unhbox\onelinebox
}%
\fi
\else
\hbox to \myboxwidth{%
$\Pi_{19,54}$ spans $L_{16.13}$%
\hfil}%
\hbox to \myboxwidth{%
$3632223636322232236$%
\hfil}%
\box\matricesbox
\fi
}%
\hfill\discretionary{}{}{}%
\setbox\matricesbox=\hbox{%
{$\left[\!\llap{\phantom{%
\begingroup \smaller\smaller\smaller
\endgroup%
}}\!\right]$}%
}%
\ifdim\wd\matricesbox>\halfwidth\myboxwidth=\hsize\else\myboxwidth=\halfwidth\fi
\vbox{%
\ifdim\myboxwidth=\hsize
\setbox\onelinebox=\hbox{%
\vbox{\hbox{%
$\Pi_{19,55}$ spans $L_{16.13}$%
}\hbox{%
$3632223636322322322$%
}%
}%
\hfill\copy\matricesbox
}%
\ifdim\wd\onelinebox>\myboxwidth
\hbox to \myboxwidth{%
$\Pi_{19,55}$ spans $L_{16.13}$%
\hfil
$3632223636322322322$%
}%
\box\matricesbox
\else
\hbox to \myboxwidth{%
\unhbox\onelinebox
}%
\fi
\else
\hbox to \myboxwidth{%
$\Pi_{19,55}$ spans $L_{16.13}$%
\hfil}%
\hbox to \myboxwidth{%
$3632223636322322322$%
\hfil}%
\box\matricesbox
\fi
}%
\hfill\discretionary{}{}{}%
\setbox\matricesbox=\hbox{%
{$\left[\!\llap{\phantom{%
\begingroup \smaller\smaller\smaller
\endgroup%
}}\!\right]$}%
}%
\ifdim\wd\matricesbox>\halfwidth\myboxwidth=\hsize\else\myboxwidth=\halfwidth\fi
\vbox{%
\ifdim\myboxwidth=\hsize
\setbox\onelinebox=\hbox{%
\vbox{\hbox{%
$\Pi_{19,56}$ spans $L_{16.13}$%
}\hbox{%
$3632223636363223222$%
}%
}%
\hfill\copy\matricesbox
}%
\ifdim\wd\onelinebox>\myboxwidth
\hbox to \myboxwidth{%
$\Pi_{19,56}$ spans $L_{16.13}$%
\hfil
$3632223636363223222$%
}%
\box\matricesbox
\else
\hbox to \myboxwidth{%
\unhbox\onelinebox
}%
\fi
\else
\hbox to \myboxwidth{%
$\Pi_{19,56}$ spans $L_{16.13}$%
\hfil}%
\hbox to \myboxwidth{%
$3632223636363223222$%
\hfil}%
\box\matricesbox
\fi
}%
\hfill\discretionary{}{}{}%
\setbox\matricesbox=\hbox{%
{$\left[\!\llap{\phantom{%
\begingroup \smaller\smaller\smaller
\endgroup%
}}\!\right]$}%
}%
\ifdim\wd\matricesbox>\halfwidth\myboxwidth=\hsize\else\myboxwidth=\halfwidth\fi
\vbox{%
\ifdim\myboxwidth=\hsize
\setbox\onelinebox=\hbox{%
\vbox{\hbox{%
$\Pi_{19,57}$ spans $L_{16.13}$%
}\hbox{%
$3632223636363632222$%
}%
}%
\hfill\copy\matricesbox
}%
\ifdim\wd\onelinebox>\myboxwidth
\hbox to \myboxwidth{%
$\Pi_{19,57}$ spans $L_{16.13}$%
\hfil
$3632223636363632222$%
}%
\box\matricesbox
\else
\hbox to \myboxwidth{%
\unhbox\onelinebox
}%
\fi
\else
\hbox to \myboxwidth{%
$\Pi_{19,57}$ spans $L_{16.13}$%
\hfil}%
\hbox to \myboxwidth{%
$3632223636363632222$%
\hfil}%
\box\matricesbox
\fi
}%
\hfill\discretionary{}{}{}%
\setbox\matricesbox=\hbox{%
{$\left[\!\llap{\phantom{%
\begingroup \smaller\smaller\smaller
\endgroup%
}}\!\right]$}%
}%
\ifdim\wd\matricesbox>\halfwidth\myboxwidth=\hsize\else\myboxwidth=\halfwidth\fi
\vbox{%
\ifdim\myboxwidth=\hsize
\setbox\onelinebox=\hbox{%
\vbox{\hbox{%
$\Pi_{19,58}$ spans $L_{16.13}$%
}\hbox{%
$3632232222236363636$%
}%
}%
\hfill\copy\matricesbox
}%
\ifdim\wd\onelinebox>\myboxwidth
\hbox to \myboxwidth{%
$\Pi_{19,58}$ spans $L_{16.13}$%
\hfil
$3632232222236363636$%
}%
\box\matricesbox
\else
\hbox to \myboxwidth{%
\unhbox\onelinebox
}%
\fi
\else
\hbox to \myboxwidth{%
$\Pi_{19,58}$ spans $L_{16.13}$%
\hfil}%
\hbox to \myboxwidth{%
$3632232222236363636$%
\hfil}%
\box\matricesbox
\fi
}%
\hfill\discretionary{}{}{}%
\setbox\matricesbox=\hbox{%
{$\left[\!\llap{\phantom{%
\begingroup \smaller\smaller\smaller
\endgroup%
}}\!\right]$}%
}%
\ifdim\wd\matricesbox>\halfwidth\myboxwidth=\hsize\else\myboxwidth=\halfwidth\fi
\vbox{%
\ifdim\myboxwidth=\hsize
\setbox\onelinebox=\hbox{%
\vbox{\hbox{%
$\Pi_{19,59}$ spans $L_{16.13}$%
}\hbox{%
$3632232223223223636$%
}%
}%
\hfill\copy\matricesbox
}%
\ifdim\wd\onelinebox>\myboxwidth
\hbox to \myboxwidth{%
$\Pi_{19,59}$ spans $L_{16.13}$%
\hfil
$3632232223223223636$%
}%
\box\matricesbox
\else
\hbox to \myboxwidth{%
\unhbox\onelinebox
}%
\fi
\else
\hbox to \myboxwidth{%
$\Pi_{19,59}$ spans $L_{16.13}$%
\hfil}%
\hbox to \myboxwidth{%
$3632232223223223636$%
\hfil}%
\box\matricesbox
\fi
}%
\hfill\discretionary{}{}{}%
\setbox\matricesbox=\hbox{%
{$\left[\!\llap{\phantom{%
\begingroup \smaller\smaller\smaller
\endgroup%
}}\!\right]$}%
}%
\ifdim\wd\matricesbox>\halfwidth\myboxwidth=\hsize\else\myboxwidth=\halfwidth\fi
\vbox{%
\ifdim\myboxwidth=\hsize
\setbox\onelinebox=\hbox{%
\vbox{\hbox{%
$\Pi_{19,60}$ spans $L_{16.13}$%
}\hbox{%
$3632232223223632236$%
}%
}%
\hfill\copy\matricesbox
}%
\ifdim\wd\onelinebox>\myboxwidth
\hbox to \myboxwidth{%
$\Pi_{19,60}$ spans $L_{16.13}$%
\hfil
$3632232223223632236$%
}%
\box\matricesbox
\else
\hbox to \myboxwidth{%
\unhbox\onelinebox
}%
\fi
\else
\hbox to \myboxwidth{%
$\Pi_{19,60}$ spans $L_{16.13}$%
\hfil}%
\hbox to \myboxwidth{%
$3632232223223632236$%
\hfil}%
\box\matricesbox
\fi
}%
\hfill\discretionary{}{}{}%
\setbox\matricesbox=\hbox{%
{$\left[\!\llap{\phantom{%
\begingroup \smaller\smaller\smaller
\endgroup%
}}\!\right]$}%
}%
\ifdim\wd\matricesbox>\halfwidth\myboxwidth=\hsize\else\myboxwidth=\halfwidth\fi
\vbox{%
\ifdim\myboxwidth=\hsize
\setbox\onelinebox=\hbox{%
\vbox{\hbox{%
$\Pi_{19,61}$ spans $L_{16.13}$%
}\hbox{%
$3632232223223636322$%
}%
}%
\hfill\copy\matricesbox
}%
\ifdim\wd\onelinebox>\myboxwidth
\hbox to \myboxwidth{%
$\Pi_{19,61}$ spans $L_{16.13}$%
\hfil
$3632232223223636322$%
}%
\box\matricesbox
\else
\hbox to \myboxwidth{%
\unhbox\onelinebox
}%
\fi
\else
\hbox to \myboxwidth{%
$\Pi_{19,61}$ spans $L_{16.13}$%
\hfil}%
\hbox to \myboxwidth{%
$3632232223223636322$%
\hfil}%
\box\matricesbox
\fi
}%
\hfill\discretionary{}{}{}%
\setbox\matricesbox=\hbox{%
{$\left[\!\llap{\phantom{%
\begingroup \smaller\smaller\smaller
\endgroup%
}}\!\right]$}%
}%
\ifdim\wd\matricesbox>\halfwidth\myboxwidth=\hsize\else\myboxwidth=\halfwidth\fi
\vbox{%
\ifdim\myboxwidth=\hsize
\setbox\onelinebox=\hbox{%
\vbox{\hbox{%
$\Pi_{19,62}$ spans $L_{16.13}$%
}\hbox{%
$3632232223632232236$%
}%
}%
\hfill\copy\matricesbox
}%
\ifdim\wd\onelinebox>\myboxwidth
\hbox to \myboxwidth{%
$\Pi_{19,62}$ spans $L_{16.13}$%
\hfil
$3632232223632232236$%
}%
\box\matricesbox
\else
\hbox to \myboxwidth{%
\unhbox\onelinebox
}%
\fi
\else
\hbox to \myboxwidth{%
$\Pi_{19,62}$ spans $L_{16.13}$%
\hfil}%
\hbox to \myboxwidth{%
$3632232223632232236$%
\hfil}%
\box\matricesbox
\fi
}%
\hfill\discretionary{}{}{}%
\setbox\matricesbox=\hbox{%
{$\left[\!\llap{\phantom{%
\begingroup \smaller\smaller\smaller
\endgroup%
}}\!\right]$}%
}%
\ifdim\wd\matricesbox>\halfwidth\myboxwidth=\hsize\else\myboxwidth=\halfwidth\fi
\vbox{%
\ifdim\myboxwidth=\hsize
\setbox\onelinebox=\hbox{%
\vbox{\hbox{%
$\Pi_{19,63}$ spans $L_{16.13}$%
}\hbox{%
$3632232223632236322$%
}%
}%
\hfill\copy\matricesbox
}%
\ifdim\wd\onelinebox>\myboxwidth
\hbox to \myboxwidth{%
$\Pi_{19,63}$ spans $L_{16.13}$%
\hfil
$3632232223632236322$%
}%
\box\matricesbox
\else
\hbox to \myboxwidth{%
\unhbox\onelinebox
}%
\fi
\else
\hbox to \myboxwidth{%
$\Pi_{19,63}$ spans $L_{16.13}$%
\hfil}%
\hbox to \myboxwidth{%
$3632232223632236322$%
\hfil}%
\box\matricesbox
\fi
}%
\hfill\discretionary{}{}{}%
\setbox\matricesbox=\hbox{%
{$\left[\!\llap{\phantom{%
\begingroup \smaller\smaller\smaller
\endgroup%
}}\!\right]$}%
}%
\ifdim\wd\matricesbox>\halfwidth\myboxwidth=\hsize\else\myboxwidth=\halfwidth\fi
\vbox{%
\ifdim\myboxwidth=\hsize
\setbox\onelinebox=\hbox{%
\vbox{\hbox{%
$\Pi_{19,64}$ spans $L_{16.13}$%
}\hbox{%
$3632232223636322322$%
}%
}%
\hfill\copy\matricesbox
}%
\ifdim\wd\onelinebox>\myboxwidth
\hbox to \myboxwidth{%
$\Pi_{19,64}$ spans $L_{16.13}$%
\hfil
$3632232223636322322$%
}%
\box\matricesbox
\else
\hbox to \myboxwidth{%
\unhbox\onelinebox
}%
\fi
\else
\hbox to \myboxwidth{%
$\Pi_{19,64}$ spans $L_{16.13}$%
\hfil}%
\hbox to \myboxwidth{%
$3632232223636322322$%
\hfil}%
\box\matricesbox
\fi
}%
\hfill\discretionary{}{}{}%
\setbox\matricesbox=\hbox{%
{$\left[\!\llap{\phantom{%
\begingroup \smaller\smaller\smaller
\endgroup%
}}\!\right]$}%
}%
\ifdim\wd\matricesbox>\halfwidth\myboxwidth=\hsize\else\myboxwidth=\halfwidth\fi
\vbox{%
\ifdim\myboxwidth=\hsize
\setbox\onelinebox=\hbox{%
\vbox{\hbox{%
$\Pi_{19,65}$ spans $L_{16.13}$%
}\hbox{%
$3632232232223223636$%
}%
}%
\hfill\copy\matricesbox
}%
\ifdim\wd\onelinebox>\myboxwidth
\hbox to \myboxwidth{%
$\Pi_{19,65}$ spans $L_{16.13}$%
\hfil
$3632232232223223636$%
}%
\box\matricesbox
\else
\hbox to \myboxwidth{%
\unhbox\onelinebox
}%
\fi
\else
\hbox to \myboxwidth{%
$\Pi_{19,65}$ spans $L_{16.13}$%
\hfil}%
\hbox to \myboxwidth{%
$3632232232223223636$%
\hfil}%
\box\matricesbox
\fi
}%
\hfill\discretionary{}{}{}%
\setbox\matricesbox=\hbox{%
{$\left[\!\llap{\phantom{%
\begingroup \smaller\smaller\smaller
\endgroup%
}}\!\right]$}%
}%
\ifdim\wd\matricesbox>\halfwidth\myboxwidth=\hsize\else\myboxwidth=\halfwidth\fi
\vbox{%
\ifdim\myboxwidth=\hsize
\setbox\onelinebox=\hbox{%
\vbox{\hbox{%
$\Pi_{19,66}$ spans $L_{16.13}$%
}\hbox{%
$3632232232223632236$%
}%
}%
\hfill\copy\matricesbox
}%
\ifdim\wd\onelinebox>\myboxwidth
\hbox to \myboxwidth{%
$\Pi_{19,66}$ spans $L_{16.13}$%
\hfil
$3632232232223632236$%
}%
\box\matricesbox
\else
\hbox to \myboxwidth{%
\unhbox\onelinebox
}%
\fi
\else
\hbox to \myboxwidth{%
$\Pi_{19,66}$ spans $L_{16.13}$%
\hfil}%
\hbox to \myboxwidth{%
$3632232232223632236$%
\hfil}%
\box\matricesbox
\fi
}%
\hfill\discretionary{}{}{}%
\setbox\matricesbox=\hbox{%
{$\left[\!\llap{\phantom{%
\begingroup \smaller\smaller\smaller
\endgroup%
}}\!\right]$}%
}%
\ifdim\wd\matricesbox>\halfwidth\myboxwidth=\hsize\else\myboxwidth=\halfwidth\fi
\vbox{%
\ifdim\myboxwidth=\hsize
\setbox\onelinebox=\hbox{%
\vbox{\hbox{%
$\Pi_{19,67}$ spans $L_{16.13}$%
}\hbox{%
$3632232232223636322$%
}%
}%
\hfill\copy\matricesbox
}%
\ifdim\wd\onelinebox>\myboxwidth
\hbox to \myboxwidth{%
$\Pi_{19,67}$ spans $L_{16.13}$%
\hfil
$3632232232223636322$%
}%
\box\matricesbox
\else
\hbox to \myboxwidth{%
\unhbox\onelinebox
}%
\fi
\else
\hbox to \myboxwidth{%
$\Pi_{19,67}$ spans $L_{16.13}$%
\hfil}%
\hbox to \myboxwidth{%
$3632232232223636322$%
\hfil}%
\box\matricesbox
\fi
}%
\hfill\discretionary{}{}{}%
\setbox\matricesbox=\hbox{%
{$\left[\!\llap{\phantom{%
\begingroup \smaller\smaller\smaller
\endgroup%
}}\!\right]$}%
}%
\ifdim\wd\matricesbox>\halfwidth\myboxwidth=\hsize\else\myboxwidth=\halfwidth\fi
\vbox{%
\ifdim\myboxwidth=\hsize
\setbox\onelinebox=\hbox{%
\vbox{\hbox{%
$\Pi_{19,68}$ spans $L_{16.13}$%
}\hbox{%
$3632232232232223636$%
}%
}%
\hfill\copy\matricesbox
}%
\ifdim\wd\onelinebox>\myboxwidth
\hbox to \myboxwidth{%
$\Pi_{19,68}$ spans $L_{16.13}$%
\hfil
$3632232232232223636$%
}%
\box\matricesbox
\else
\hbox to \myboxwidth{%
\unhbox\onelinebox
}%
\fi
\else
\hbox to \myboxwidth{%
$\Pi_{19,68}$ spans $L_{16.13}$%
\hfil}%
\hbox to \myboxwidth{%
$3632232232232223636$%
\hfil}%
\box\matricesbox
\fi
}%
\hfill\discretionary{}{}{}%
\setbox\matricesbox=\hbox{%
{$\left[\!\llap{\phantom{%
\begingroup \smaller\smaller\smaller
\endgroup%
}}\!\right]$}%
}%
\ifdim\wd\matricesbox>\halfwidth\myboxwidth=\hsize\else\myboxwidth=\halfwidth\fi
\vbox{%
\ifdim\myboxwidth=\hsize
\setbox\onelinebox=\hbox{%
\vbox{\hbox{%
$\Pi_{19,69}$ spans $L_{16.13}$%
}\hbox{%
$3632232236322223636$%
}%
}%
\hfill\copy\matricesbox
}%
\ifdim\wd\onelinebox>\myboxwidth
\hbox to \myboxwidth{%
$\Pi_{19,69}$ spans $L_{16.13}$%
\hfil
$3632232236322223636$%
}%
\box\matricesbox
\else
\hbox to \myboxwidth{%
\unhbox\onelinebox
}%
\fi
\else
\hbox to \myboxwidth{%
$\Pi_{19,69}$ spans $L_{16.13}$%
\hfil}%
\hbox to \myboxwidth{%
$3632232236322223636$%
\hfil}%
\box\matricesbox
\fi
}%
\hfill\discretionary{}{}{}%
\setbox\matricesbox=\hbox{%
{$\left[\!\llap{\phantom{%
\begingroup \smaller\smaller\smaller
\endgroup%
}}\!\right]$}%
}%
\ifdim\wd\matricesbox>\halfwidth\myboxwidth=\hsize\else\myboxwidth=\halfwidth\fi
\vbox{%
\ifdim\myboxwidth=\hsize
\setbox\onelinebox=\hbox{%
\vbox{\hbox{%
$\Pi_{19,70}$ spans $L_{16.13}$%
}\hbox{%
$3632236322223223636$%
}%
}%
\hfill\copy\matricesbox
}%
\ifdim\wd\onelinebox>\myboxwidth
\hbox to \myboxwidth{%
$\Pi_{19,70}$ spans $L_{16.13}$%
\hfil
$3632236322223223636$%
}%
\box\matricesbox
\else
\hbox to \myboxwidth{%
\unhbox\onelinebox
}%
\fi
\else
\hbox to \myboxwidth{%
$\Pi_{19,70}$ spans $L_{16.13}$%
\hfil}%
\hbox to \myboxwidth{%
$3632236322223223636$%
\hfil}%
\box\matricesbox
\fi
}%
\hfill\discretionary{}{}{}%
\setbox\matricesbox=\hbox{%
{$\left[\!\llap{\phantom{%
\begingroup \smaller\smaller\smaller
\endgroup%
}}\!\right]$}%
}%
\ifdim\wd\matricesbox>\halfwidth\myboxwidth=\hsize\else\myboxwidth=\halfwidth\fi
\vbox{%
\ifdim\myboxwidth=\hsize
\setbox\onelinebox=\hbox{%
\vbox{\hbox{%
$\Pi_{19,71}$ spans $L_{16.13}$%
}\hbox{%
$3632236322223636322$%
}%
}%
\hfill\copy\matricesbox
}%
\ifdim\wd\onelinebox>\myboxwidth
\hbox to \myboxwidth{%
$\Pi_{19,71}$ spans $L_{16.13}$%
\hfil
$3632236322223636322$%
}%
\box\matricesbox
\else
\hbox to \myboxwidth{%
\unhbox\onelinebox
}%
\fi
\else
\hbox to \myboxwidth{%
$\Pi_{19,71}$ spans $L_{16.13}$%
\hfil}%
\hbox to \myboxwidth{%
$3632236322223636322$%
\hfil}%
\box\matricesbox
\fi
}%
\hfill\discretionary{}{}{}%
\setbox\matricesbox=\hbox{%
{$\left[\!\llap{\phantom{%
\begingroup \smaller\smaller\smaller
\endgroup%
}}\!\right]$}%
}%
\ifdim\wd\matricesbox>\halfwidth\myboxwidth=\hsize\else\myboxwidth=\halfwidth\fi
\vbox{%
\ifdim\myboxwidth=\hsize
\setbox\onelinebox=\hbox{%
\vbox{\hbox{%
$\Pi_{19,72}$ spans $L_{16.13}$%
}\hbox{%
$3632236322232223636$%
}%
}%
\hfill\copy\matricesbox
}%
\ifdim\wd\onelinebox>\myboxwidth
\hbox to \myboxwidth{%
$\Pi_{19,72}$ spans $L_{16.13}$%
\hfil
$3632236322232223636$%
}%
\box\matricesbox
\else
\hbox to \myboxwidth{%
\unhbox\onelinebox
}%
\fi
\else
\hbox to \myboxwidth{%
$\Pi_{19,72}$ spans $L_{16.13}$%
\hfil}%
\hbox to \myboxwidth{%
$3632236322232223636$%
\hfil}%
\box\matricesbox
\fi
}%
\hfill\discretionary{}{}{}%
\setbox\matricesbox=\hbox{%
{$\left[\!\llap{\phantom{%
\begingroup \smaller\smaller\smaller
\endgroup%
}}\!\right]$}%
}%
\ifdim\wd\matricesbox>\halfwidth\myboxwidth=\hsize\else\myboxwidth=\halfwidth\fi
\vbox{%
\ifdim\myboxwidth=\hsize
\setbox\onelinebox=\hbox{%
\vbox{\hbox{%
$\Pi_{19,73}$ spans $L_{16.13}$%
}\hbox{%
$3632236322322223636$%
}%
}%
\hfill\copy\matricesbox
}%
\ifdim\wd\onelinebox>\myboxwidth
\hbox to \myboxwidth{%
$\Pi_{19,73}$ spans $L_{16.13}$%
\hfil
$3632236322322223636$%
}%
\box\matricesbox
\else
\hbox to \myboxwidth{%
\unhbox\onelinebox
}%
\fi
\else
\hbox to \myboxwidth{%
$\Pi_{19,73}$ spans $L_{16.13}$%
\hfil}%
\hbox to \myboxwidth{%
$3632236322322223636$%
\hfil}%
\box\matricesbox
\fi
}%
\hfill\discretionary{}{}{}%
\setbox\matricesbox=\hbox{%
{$\left[\!\llap{\phantom{%
\begingroup \smaller\smaller\smaller
\endgroup%
}}\!\right]$}%
}%
\ifdim\wd\matricesbox>\halfwidth\myboxwidth=\hsize\else\myboxwidth=\halfwidth\fi
\vbox{%
\ifdim\myboxwidth=\hsize
\setbox\onelinebox=\hbox{%
\vbox{\hbox{%
$\Pi_{19,74}$ spans $L_{16.13}$%
}\hbox{%
$3636322223223223636$%
}%
}%
\hfill\copy\matricesbox
}%
\ifdim\wd\onelinebox>\myboxwidth
\hbox to \myboxwidth{%
$\Pi_{19,74}$ spans $L_{16.13}$%
\hfil
$3636322223223223636$%
}%
\box\matricesbox
\else
\hbox to \myboxwidth{%
\unhbox\onelinebox
}%
\fi
\else
\hbox to \myboxwidth{%
$\Pi_{19,74}$ spans $L_{16.13}$%
\hfil}%
\hbox to \myboxwidth{%
$3636322223223223636$%
\hfil}%
\box\matricesbox
\fi
}%
\hfill\discretionary{}{}{}%
\setbox\matricesbox=\hbox{%
{$\left[\!\llap{\phantom{%
\begingroup \smaller\smaller\smaller
\endgroup%
}}\!\right]$}%
}%
\ifdim\wd\matricesbox>\halfwidth\myboxwidth=\hsize\else\myboxwidth=\halfwidth\fi
\vbox{%
\ifdim\myboxwidth=\hsize
\setbox\onelinebox=\hbox{%
\vbox{\hbox{%
$\Pi_{19,75}$ spans $L_{16.13}$%
}\hbox{%
$3636322223223636322$%
}%
}%
\hfill\copy\matricesbox
}%
\ifdim\wd\onelinebox>\myboxwidth
\hbox to \myboxwidth{%
$\Pi_{19,75}$ spans $L_{16.13}$%
\hfil
$3636322223223636322$%
}%
\box\matricesbox
\else
\hbox to \myboxwidth{%
\unhbox\onelinebox
}%
\fi
\else
\hbox to \myboxwidth{%
$\Pi_{19,75}$ spans $L_{16.13}$%
\hfil}%
\hbox to \myboxwidth{%
$3636322223223636322$%
\hfil}%
\box\matricesbox
\fi
}%
\hfill\discretionary{}{}{}%
\setbox\matricesbox=\hbox{%
{$\left[\!\llap{\phantom{%
\begingroup \smaller\smaller\smaller
\endgroup%
}}\!\right]$}%
}%
\ifdim\wd\matricesbox>\halfwidth\myboxwidth=\hsize\else\myboxwidth=\halfwidth\fi
\vbox{%
\ifdim\myboxwidth=\hsize
\setbox\onelinebox=\hbox{%
\vbox{\hbox{%
$\Pi_{19,76}$ spans $L_{16.13}$%
}\hbox{%
$3636322232223223636$%
}%
}%
\hfill\copy\matricesbox
}%
\ifdim\wd\onelinebox>\myboxwidth
\hbox to \myboxwidth{%
$\Pi_{19,76}$ spans $L_{16.13}$%
\hfil
$3636322232223223636$%
}%
\box\matricesbox
\else
\hbox to \myboxwidth{%
\unhbox\onelinebox
}%
\fi
\else
\hbox to \myboxwidth{%
$\Pi_{19,76}$ spans $L_{16.13}$%
\hfil}%
\hbox to \myboxwidth{%
$3636322232223223636$%
\hfil}%
\box\matricesbox
\fi
}%
\hfill\discretionary{}{}{}%
\setbox\matricesbox=\hbox{%
{$\left[\!\llap{\phantom{%
\begingroup \smaller\smaller\smaller
\endgroup%
}}\!\right]$}%
}%
\ifdim\wd\matricesbox>\halfwidth\myboxwidth=\hsize\else\myboxwidth=\halfwidth\fi
\vbox{%
\ifdim\myboxwidth=\hsize
\setbox\onelinebox=\hbox{%
\vbox{\hbox{%
$\Pi_{19,77}$ spans $L_{142.20}$%
}\hbox{%
$\infty4224\infty4224\infty4224\infty422$%
}%
}%
\hfill\copy\matricesbox
}%
\ifdim\wd\onelinebox>\myboxwidth
\hbox to \myboxwidth{%
$\Pi_{19,77}$ spans $L_{142.20}$%
\hfil
$\infty4224\infty4224\infty4224\infty422$%
}%
\box\matricesbox
\else
\hbox to \myboxwidth{%
\unhbox\onelinebox
}%
\fi
\else
\hbox to \myboxwidth{%
$\Pi_{19,77}$ spans $L_{142.20}$%
\hfil}%
\hbox to \myboxwidth{%
$\infty4224\infty4224\infty4224\infty422$%
\hfil}%
\box\matricesbox
\fi
}%
\hfill\discretionary{}{}{}%
\setbox\matricesbox=\hbox{%
{$\left[\!\llap{\phantom{%
\begingroup \smaller\smaller\smaller
\endgroup%
}}\!\right]$}%
}%
\ifdim\wd\matricesbox>\halfwidth\myboxwidth=\hsize\else\myboxwidth=\halfwidth\fi
\vbox{%
\ifdim\myboxwidth=\hsize
\setbox\onelinebox=\hbox{%
\vbox{\hbox{%
$\Pi_{19,78}$ spans $L_{251.3}$%
}\hbox{%
$3223223223223222622$%
}%
}%
\hfill\copy\matricesbox
}%
\ifdim\wd\onelinebox>\myboxwidth
\hbox to \myboxwidth{%
$\Pi_{19,78}$ spans $L_{251.3}$%
\hfil
$3223223223223222622$%
}%
\box\matricesbox
\else
\hbox to \myboxwidth{%
\unhbox\onelinebox
}%
\fi
\else
\hbox to \myboxwidth{%
$\Pi_{19,78}$ spans $L_{251.3}$%
\hfil}%
\hbox to \myboxwidth{%
$3223223223223222622$%
\hfil}%
\box\matricesbox
\fi
}%
\hfill\discretionary{}{}{}%

\vskip2pt\hrule\vskip2pt

\leavevmode\setbox\matricesbox=\hbox{%
{$\left[\!\llap{\phantom{%
\begingroup \smaller\smaller\smaller\begin{tabular}{@{}c@{}}%
\phantom{0}\\\phantom{0}\\\phantom{0}
\end{tabular}\endgroup%
}}\right.$}%
\begingroup \smaller\smaller\smaller\begin{tabular}{@{}c@{}}%
-4\\\phantom{0}\\\phantom{0}
\end{tabular}\endgroup%
\kern3pt%
\begingroup \smaller\smaller\smaller\begin{tabular}{@{}c@{}}%
\phantom{0}\\9\\\phantom{0}
\end{tabular}\endgroup%
\kern3pt%
\begingroup \smaller\smaller\smaller\begin{tabular}{@{}c@{}}%
\phantom{0}\\\phantom{0}\\9
\end{tabular}\endgroup%
{$\left.\llap{\phantom{%
\begingroup \smaller\smaller\smaller\begin{tabular}{@{}c@{}}%
\phantom{0}\\\phantom{0}\\\phantom{0}
\end{tabular}\endgroup%
}}\!\right]$}%
{$\left[\!\llap{\phantom{%
\begingroup \smaller\smaller\smaller\begin{tabular}{@{}c@{}}%
0\\0\\0
\end{tabular}\endgroup%
}}\right.$}%
\begingroup \smaller\smaller\smaller\begin{tabular}{@{}c@{}}%
2\\1\\-1
\end{tabular}\endgroup%
\kern3pt%
\begingroup \smaller\smaller\smaller\begin{tabular}{@{}c@{}}%
18\\11\\-5
\end{tabular}\endgroup%
\kern3pt%
\begingroup \smaller\smaller\smaller\begin{tabular}{@{}c@{}}%
9\\6\\-1
\end{tabular}\endgroup%
{$\left.\llap{\phantom{%
\begingroup \smaller\smaller\smaller\begin{tabular}{@{}c@{}}%
0\\0\\0
\end{tabular}\endgroup%
}}\!\right]$}%
}%
\ifdim\wd\matricesbox>\halfwidth\myboxwidth=\hsize\else\myboxwidth=\halfwidth\fi
\vbox{%
\ifdim\myboxwidth=\hsize
\setbox\onelinebox=\hbox{%
\vbox{\hbox{%
$\Pi_{20,1}$ spans $L_{216.2}$%
}\hbox{%
$42|24\slashinfty42|24\slashinfty42|24\slashinfty42|24\slashinfty\rtimes D_{8}$%
}%
}%
\hfill\copy\matricesbox
}%
\ifdim\wd\onelinebox>\myboxwidth
\hbox to \myboxwidth{%
$\Pi_{20,1}$ spans $L_{216.2}$%
\hfil
$42|24\slashinfty42|24\slashinfty42|24\slashinfty42|24\slashinfty\rtimes D_{8}$%
}%
\box\matricesbox
\else
\hbox to \myboxwidth{%
\unhbox\onelinebox
}%
\fi
\else
\hbox to \myboxwidth{%
$\Pi_{20,1}$ spans $L_{216.2}$%
\hfil}%
\hbox to \myboxwidth{%
$42|24\slashinfty42|24\slashinfty42|24\slashinfty42|24\slashinfty\rtimes D_{8}$%
\hfil}%
\box\matricesbox
\fi
}%
\hfill\discretionary{}{}{}%
\setbox\matricesbox=\hbox{%
{$\left[\!\llap{\phantom{%
\begingroup \smaller\smaller\smaller\begin{tabular}{@{}c@{}}%
\phantom{0}\\\phantom{0}\\\phantom{0}
\end{tabular}\endgroup%
}}\right.$}%
\begingroup \smaller\smaller\smaller\begin{tabular}{@{}c@{}}%
-3\\\phantom{0}\\\phantom{0}
\end{tabular}\endgroup%
\kern3pt%
\begingroup \smaller\smaller\smaller\begin{tabular}{@{}c@{}}%
\phantom{0}\\15/2\\\phantom{0}
\end{tabular}\endgroup%
\kern3pt%
\begingroup \smaller\smaller\smaller\begin{tabular}{@{}c@{}}%
\phantom{0}\\\phantom{0}\\5/2
\end{tabular}\endgroup%
{$\left.\llap{\phantom{%
\begingroup \smaller\smaller\smaller\begin{tabular}{@{}c@{}}%
\phantom{0}\\\phantom{0}\\\phantom{0}
\end{tabular}\endgroup%
}}\!\right]$}%
{$\left[\!\llap{\phantom{%
\begingroup \smaller\smaller\smaller\begin{tabular}{@{}c@{}}%
0\\0\\0
\end{tabular}\endgroup%
}}\right.$}%
\begingroup \smaller\smaller\smaller\begin{tabular}{@{}c@{}}%
30\\-19\\3
\end{tabular}\endgroup%
\kern3pt%
\begingroup \smaller\smaller\smaller\begin{tabular}{@{}c@{}}%
10\\-6\\4
\end{tabular}\endgroup%
\kern3pt%
\begingroup \smaller\smaller\smaller\begin{tabular}{@{}c@{}}%
10\\-5\\7
\end{tabular}\endgroup%
\kern3pt%
\begingroup \smaller\smaller\smaller\begin{tabular}{@{}c@{}}%
3\\-1\\3
\end{tabular}\endgroup%
\kern3pt%
\begingroup \smaller\smaller\smaller\begin{tabular}{@{}c@{}}%
10\\-1\\11
\end{tabular}\endgroup%
{$\left.\llap{\phantom{%
\begingroup \smaller\smaller\smaller\begin{tabular}{@{}c@{}}%
0\\0\\0
\end{tabular}\endgroup%
}}\!\right]$}%
}%
\ifdim\wd\matricesbox>\halfwidth\myboxwidth=\hsize\else\myboxwidth=\halfwidth\fi
\vbox{%
\ifdim\myboxwidth=\hsize
\setbox\onelinebox=\hbox{%
\vbox{\hbox{%
$\Pi_{20,2}$ spans $L_{16.9}$%
}\hbox{%
$\slashthree6322\slashthree2236\slashthree6322\slashthree2236\rtimes D_{4}$%
}%
}%
\hfill\copy\matricesbox
}%
\ifdim\wd\onelinebox>\myboxwidth
\hbox to \myboxwidth{%
$\Pi_{20,2}$ spans $L_{16.9}$%
\hfil
$\slashthree6322\slashthree2236\slashthree6322\slashthree2236\rtimes D_{4}$%
}%
\box\matricesbox
\else
\hbox to \myboxwidth{%
\unhbox\onelinebox
}%
\fi
\else
\hbox to \myboxwidth{%
$\Pi_{20,2}$ spans $L_{16.9}$%
\hfil}%
\hbox to \myboxwidth{%
$\slashthree6322\slashthree2236\slashthree6322\slashthree2236\rtimes D_{4}$%
\hfil}%
\box\matricesbox
\fi
}%
\hfill\discretionary{}{}{}%
\setbox\matricesbox=\hbox{%
{$\left[\!\llap{\phantom{%
\begingroup \smaller\smaller\smaller\begin{tabular}{@{}c@{}}%
\phantom{0}\\\phantom{0}\\\phantom{0}
\end{tabular}\endgroup%
}}\right.$}%
\begingroup \smaller\smaller\smaller\begin{tabular}{@{}c@{}}%
-1\\\phantom{0}\\\phantom{0}
\end{tabular}\endgroup%
\kern3pt%
\begingroup \smaller\smaller\smaller\begin{tabular}{@{}c@{}}%
\phantom{0}\\45/2\\\phantom{0}
\end{tabular}\endgroup%
\kern3pt%
\begingroup \smaller\smaller\smaller\begin{tabular}{@{}c@{}}%
\phantom{0}\\\phantom{0}\\15/2
\end{tabular}\endgroup%
{$\left.\llap{\phantom{%
\begingroup \smaller\smaller\smaller\begin{tabular}{@{}c@{}}%
\phantom{0}\\\phantom{0}\\\phantom{0}
\end{tabular}\endgroup%
}}\!\right]$}%
{$\left[\!\llap{\phantom{%
\begingroup \smaller\smaller\smaller\begin{tabular}{@{}c@{}}%
0\\0\\0
\end{tabular}\endgroup%
}}\right.$}%
\begingroup \smaller\smaller\smaller\begin{tabular}{@{}c@{}}%
9\\-2\\0
\end{tabular}\endgroup%
\kern3pt%
\begingroup \smaller\smaller\smaller\begin{tabular}{@{}c@{}}%
30\\-6\\4
\end{tabular}\endgroup%
\kern3pt%
\begingroup \smaller\smaller\smaller\begin{tabular}{@{}c@{}}%
30\\-5\\7
\end{tabular}\endgroup%
\kern3pt%
\begingroup \smaller\smaller\smaller\begin{tabular}{@{}c@{}}%
90\\-11\\27
\end{tabular}\endgroup%
\kern3pt%
\begingroup \smaller\smaller\smaller\begin{tabular}{@{}c@{}}%
90\\-8\\30
\end{tabular}\endgroup%
\kern3pt%
\begingroup \smaller\smaller\smaller\begin{tabular}{@{}c@{}}%
5\\0\\2
\end{tabular}\endgroup%
{$\left.\llap{\phantom{%
\begingroup \smaller\smaller\smaller\begin{tabular}{@{}c@{}}%
0\\0\\0
\end{tabular}\endgroup%
}}\!\right]$}%
}%
\ifdim\wd\matricesbox>\halfwidth\myboxwidth=\hsize\else\myboxwidth=\halfwidth\fi
\vbox{%
\ifdim\myboxwidth=\hsize
\setbox\onelinebox=\hbox{%
\vbox{\hbox{%
$\Pi_{20,3}$ spans $L_{16.13}$%
}\hbox{%
$3632|23632|23632|23632|2\rtimes D_{4}$%
}%
}%
\hfill\copy\matricesbox
}%
\ifdim\wd\onelinebox>\myboxwidth
\hbox to \myboxwidth{%
$\Pi_{20,3}$ spans $L_{16.13}$%
\hfil
$3632|23632|23632|23632|2\rtimes D_{4}$%
}%
\box\matricesbox
\else
\hbox to \myboxwidth{%
\unhbox\onelinebox
}%
\fi
\else
\hbox to \myboxwidth{%
$\Pi_{20,3}$ spans $L_{16.13}$%
\hfil}%
\hbox to \myboxwidth{%
$3632|23632|23632|23632|2\rtimes D_{4}$%
\hfil}%
\box\matricesbox
\fi
}%
\hfill\discretionary{}{}{}%
\setbox\matricesbox=\hbox{%
{$\left[\!\llap{\phantom{%
\begingroup \smaller\smaller\smaller\begin{tabular}{@{}c@{}}%
\phantom{0}\\\phantom{0}\\\phantom{0}
\end{tabular}\endgroup%
}}\right.$}%
\begingroup \smaller\smaller\smaller\begin{tabular}{@{}c@{}}%
-5\\\phantom{0}\\\phantom{0}
\end{tabular}\endgroup%
\kern3pt%
\begingroup \smaller\smaller\smaller\begin{tabular}{@{}c@{}}%
\phantom{0}\\3/2\\\phantom{0}
\end{tabular}\endgroup%
\kern3pt%
\begingroup \smaller\smaller\smaller\begin{tabular}{@{}c@{}}%
\phantom{0}\\\phantom{0}\\9/2
\end{tabular}\endgroup%
{$\left.\llap{\phantom{%
\begingroup \smaller\smaller\smaller\begin{tabular}{@{}c@{}}%
\phantom{0}\\\phantom{0}\\\phantom{0}
\end{tabular}\endgroup%
}}\!\right]$}%
{$\left[\!\llap{\phantom{%
\begingroup \smaller\smaller\smaller\begin{tabular}{@{}c@{}}%
0\\0\\0
\end{tabular}\endgroup%
}}\right.$}%
\begingroup \smaller\smaller\smaller\begin{tabular}{@{}c@{}}%
6\\11\\-1
\end{tabular}\endgroup%
\kern3pt%
\begingroup \smaller\smaller\smaller\begin{tabular}{@{}c@{}}%
18\\30\\-8
\end{tabular}\endgroup%
\kern3pt%
\begingroup \smaller\smaller\smaller\begin{tabular}{@{}c@{}}%
18\\27\\-11
\end{tabular}\endgroup%
\kern3pt%
\begingroup \smaller\smaller\smaller\begin{tabular}{@{}c@{}}%
1\\1\\-1
\end{tabular}\endgroup%
\kern3pt%
\begingroup \smaller\smaller\smaller\begin{tabular}{@{}c@{}}%
18\\3\\-19
\end{tabular}\endgroup%
{$\left.\llap{\phantom{%
\begingroup \smaller\smaller\smaller\begin{tabular}{@{}c@{}}%
0\\0\\0
\end{tabular}\endgroup%
}}\!\right]$}%
}%
\ifdim\wd\matricesbox>\halfwidth\myboxwidth=\hsize\else\myboxwidth=\halfwidth\fi
\vbox{%
\ifdim\myboxwidth=\hsize
\setbox\onelinebox=\hbox{%
\vbox{\hbox{%
$\Pi_{20,4}$ spans $L_{16.7}$%
}\hbox{%
$36\slashthree6322\slashthree2236\slashthree6322\slashthree22\rtimes D_{4}$%
}%
}%
\hfill\copy\matricesbox
}%
\ifdim\wd\onelinebox>\myboxwidth
\hbox to \myboxwidth{%
$\Pi_{20,4}$ spans $L_{16.7}$%
\hfil
$36\slashthree6322\slashthree2236\slashthree6322\slashthree22\rtimes D_{4}$%
}%
\box\matricesbox
\else
\hbox to \myboxwidth{%
\unhbox\onelinebox
}%
\fi
\else
\hbox to \myboxwidth{%
$\Pi_{20,4}$ spans $L_{16.7}$%
\hfil}%
\hbox to \myboxwidth{%
$36\slashthree6322\slashthree2236\slashthree6322\slashthree22\rtimes D_{4}$%
\hfil}%
\box\matricesbox
\fi
}%
\hfill\discretionary{}{}{}%
\setbox\matricesbox=\hbox{%
{$\left[\!\llap{\phantom{%
\begingroup \smaller\smaller\smaller\begin{tabular}{@{}c@{}}%
\phantom{0}\\\phantom{0}\\\phantom{0}
\end{tabular}\endgroup%
}}\right.$}%
\begingroup \smaller\smaller\smaller\begin{tabular}{@{}c@{}}%
-1\\\phantom{0}\\\phantom{0}
\end{tabular}\endgroup%
\kern3pt%
\begingroup \smaller\smaller\smaller\begin{tabular}{@{}c@{}}%
\phantom{0}\\45/2\\\phantom{0}
\end{tabular}\endgroup%
\kern3pt%
\begingroup \smaller\smaller\smaller\begin{tabular}{@{}c@{}}%
\phantom{0}\\\phantom{0}\\15/2
\end{tabular}\endgroup%
{$\left.\llap{\phantom{%
\begingroup \smaller\smaller\smaller\begin{tabular}{@{}c@{}}%
\phantom{0}\\\phantom{0}\\\phantom{0}
\end{tabular}\endgroup%
}}\!\right]$}%
{$\left[\!\llap{\phantom{%
\begingroup \smaller\smaller\smaller\begin{tabular}{@{}c@{}}%
0\\0\\0
\end{tabular}\endgroup%
}}\right.$}%
\begingroup \smaller\smaller\smaller\begin{tabular}{@{}c@{}}%
90\\19\\3
\end{tabular}\endgroup%
\kern3pt%
\begingroup \smaller\smaller\smaller\begin{tabular}{@{}c@{}}%
5\\1\\1
\end{tabular}\endgroup%
\kern3pt%
\begingroup \smaller\smaller\smaller\begin{tabular}{@{}c@{}}%
9\\1\\3
\end{tabular}\endgroup%
\kern3pt%
\begingroup \smaller\smaller\smaller\begin{tabular}{@{}c@{}}%
30\\1\\11
\end{tabular}\endgroup%
\kern3pt%
\begingroup \smaller\smaller\smaller\begin{tabular}{@{}c@{}}%
30\\-1\\11
\end{tabular}\endgroup%
\kern3pt%
\begingroup \smaller\smaller\smaller\begin{tabular}{@{}c@{}}%
90\\-8\\30
\end{tabular}\endgroup%
\kern3pt%
\begingroup \smaller\smaller\smaller\begin{tabular}{@{}c@{}}%
90\\-11\\27
\end{tabular}\endgroup%
\kern3pt%
\begingroup \smaller\smaller\smaller\begin{tabular}{@{}c@{}}%
30\\-5\\7
\end{tabular}\endgroup%
\kern3pt%
\begingroup \smaller\smaller\smaller\begin{tabular}{@{}c@{}}%
30\\-6\\4
\end{tabular}\endgroup%
\kern3pt%
\begingroup \smaller\smaller\smaller\begin{tabular}{@{}c@{}}%
90\\-19\\3
\end{tabular}\endgroup%
{$\left.\llap{\phantom{%
\begingroup \smaller\smaller\smaller\begin{tabular}{@{}c@{}}%
0\\0\\0
\end{tabular}\endgroup%
}}\!\right]$}%
}%
\ifdim\wd\matricesbox>\halfwidth\myboxwidth=\hsize\else\myboxwidth=\halfwidth\fi
\vbox{%
\ifdim\myboxwidth=\hsize
\setbox\onelinebox=\hbox{%
\vbox{\hbox{%
$\Pi_{20,5}$ spans $L_{16.13}$%
}\hbox{%
$363222\slashthree222363636\slashthree636\rtimes D_{2}$%
}%
}%
\hfill\copy\matricesbox
}%
\ifdim\wd\onelinebox>\myboxwidth
\hbox to \myboxwidth{%
$\Pi_{20,5}$ spans $L_{16.13}$%
\hfil
$363222\slashthree222363636\slashthree636\rtimes D_{2}$%
}%
\box\matricesbox
\else
\hbox to \myboxwidth{%
\unhbox\onelinebox
}%
\fi
\else
\hbox to \myboxwidth{%
$\Pi_{20,5}$ spans $L_{16.13}$%
\hfil}%
\hbox to \myboxwidth{%
$363222\slashthree222363636\slashthree636\rtimes D_{2}$%
\hfil}%
\box\matricesbox
\fi
}%
\hfill\discretionary{}{}{}%
\setbox\matricesbox=\hbox{%
{$\left[\!\llap{\phantom{%
\begingroup \smaller\smaller\smaller\begin{tabular}{@{}c@{}}%
\phantom{0}\\\phantom{0}\\\phantom{0}
\end{tabular}\endgroup%
}}\right.$}%
\begingroup \smaller\smaller\smaller\begin{tabular}{@{}c@{}}%
-1\\\phantom{0}\\\phantom{0}
\end{tabular}\endgroup%
\kern3pt%
\begingroup \smaller\smaller\smaller\begin{tabular}{@{}c@{}}%
\phantom{0}\\15/2\\\phantom{0}
\end{tabular}\endgroup%
\kern3pt%
\begingroup \smaller\smaller\smaller\begin{tabular}{@{}c@{}}%
\phantom{0}\\\phantom{0}\\45/2
\end{tabular}\endgroup%
{$\left.\llap{\phantom{%
\begingroup \smaller\smaller\smaller\begin{tabular}{@{}c@{}}%
\phantom{0}\\\phantom{0}\\\phantom{0}
\end{tabular}\endgroup%
}}\!\right]$}%
{$\left[\!\llap{\phantom{%
\begingroup \smaller\smaller\smaller\begin{tabular}{@{}c@{}}%
0\\0\\0
\end{tabular}\endgroup%
}}\right.$}%
\begingroup \smaller\smaller\smaller\begin{tabular}{@{}c@{}}%
30\\11\\-1
\end{tabular}\endgroup%
\kern3pt%
\begingroup \smaller\smaller\smaller\begin{tabular}{@{}c@{}}%
90\\30\\-8
\end{tabular}\endgroup%
\kern3pt%
\begingroup \smaller\smaller\smaller\begin{tabular}{@{}c@{}}%
90\\27\\-11
\end{tabular}\endgroup%
\kern3pt%
\begingroup \smaller\smaller\smaller\begin{tabular}{@{}c@{}}%
5\\1\\-1
\end{tabular}\endgroup%
\kern3pt%
\begingroup \smaller\smaller\smaller\begin{tabular}{@{}c@{}}%
9\\0\\-2
\end{tabular}\endgroup%
\kern3pt%
\begingroup \smaller\smaller\smaller\begin{tabular}{@{}c@{}}%
30\\-4\\-6
\end{tabular}\endgroup%
\kern3pt%
\begingroup \smaller\smaller\smaller\begin{tabular}{@{}c@{}}%
30\\-7\\-5
\end{tabular}\endgroup%
\kern3pt%
\begingroup \smaller\smaller\smaller\begin{tabular}{@{}c@{}}%
90\\-27\\-11
\end{tabular}\endgroup%
\kern3pt%
\begingroup \smaller\smaller\smaller\begin{tabular}{@{}c@{}}%
90\\-30\\-8
\end{tabular}\endgroup%
\kern3pt%
\begingroup \smaller\smaller\smaller\begin{tabular}{@{}c@{}}%
30\\-11\\-1
\end{tabular}\endgroup%
{$\left.\llap{\phantom{%
\begingroup \smaller\smaller\smaller\begin{tabular}{@{}c@{}}%
0\\0\\0
\end{tabular}\endgroup%
}}\!\right]$}%
}%
\ifdim\wd\matricesbox>\halfwidth\myboxwidth=\hsize\else\myboxwidth=\halfwidth\fi
\vbox{%
\ifdim\myboxwidth=\hsize
\setbox\onelinebox=\hbox{%
\vbox{\hbox{%
$\Pi_{20,6}$ spans $L_{16.13}$%
}\hbox{%
$36322236\slashthree632223636\slashthree6\rtimes D_{2}$%
}%
}%
\hfill\copy\matricesbox
}%
\ifdim\wd\onelinebox>\myboxwidth
\hbox to \myboxwidth{%
$\Pi_{20,6}$ spans $L_{16.13}$%
\hfil
$36322236\slashthree632223636\slashthree6\rtimes D_{2}$%
}%
\box\matricesbox
\else
\hbox to \myboxwidth{%
\unhbox\onelinebox
}%
\fi
\else
\hbox to \myboxwidth{%
$\Pi_{20,6}$ spans $L_{16.13}$%
\hfil}%
\hbox to \myboxwidth{%
$36322236\slashthree632223636\slashthree6\rtimes D_{2}$%
\hfil}%
\box\matricesbox
\fi
}%
\hfill\discretionary{}{}{}%
\setbox\matricesbox=\hbox{%
{$\left[\!\llap{\phantom{%
\begingroup \smaller\smaller\smaller\begin{tabular}{@{}c@{}}%
\phantom{0}\\\phantom{0}\\\phantom{0}
\end{tabular}\endgroup%
}}\right.$}%
\begingroup \smaller\smaller\smaller\begin{tabular}{@{}c@{}}%
-1\\\phantom{0}\\\phantom{0}
\end{tabular}\endgroup%
\kern3pt%
\begingroup \smaller\smaller\smaller\begin{tabular}{@{}c@{}}%
\phantom{0}\\45/2\\\phantom{0}
\end{tabular}\endgroup%
\kern3pt%
\begingroup \smaller\smaller\smaller\begin{tabular}{@{}c@{}}%
\phantom{0}\\\phantom{0}\\15/2
\end{tabular}\endgroup%
{$\left.\llap{\phantom{%
\begingroup \smaller\smaller\smaller\begin{tabular}{@{}c@{}}%
\phantom{0}\\\phantom{0}\\\phantom{0}
\end{tabular}\endgroup%
}}\!\right]$}%
{$\left[\!\llap{\phantom{%
\begingroup \smaller\smaller\smaller\begin{tabular}{@{}c@{}}%
0\\0\\0
\end{tabular}\endgroup%
}}\right.$}%
\begingroup \smaller\smaller\smaller\begin{tabular}{@{}c@{}}%
90\\19\\3
\end{tabular}\endgroup%
\kern3pt%
\begingroup \smaller\smaller\smaller\begin{tabular}{@{}c@{}}%
30\\6\\4
\end{tabular}\endgroup%
\kern3pt%
\begingroup \smaller\smaller\smaller\begin{tabular}{@{}c@{}}%
30\\5\\7
\end{tabular}\endgroup%
\kern3pt%
\begingroup \smaller\smaller\smaller\begin{tabular}{@{}c@{}}%
9\\1\\3
\end{tabular}\endgroup%
\kern3pt%
\begingroup \smaller\smaller\smaller\begin{tabular}{@{}c@{}}%
5\\0\\2
\end{tabular}\endgroup%
\kern3pt%
\begingroup \smaller\smaller\smaller\begin{tabular}{@{}c@{}}%
90\\-8\\30
\end{tabular}\endgroup%
\kern3pt%
\begingroup \smaller\smaller\smaller\begin{tabular}{@{}c@{}}%
90\\-11\\27
\end{tabular}\endgroup%
\kern3pt%
\begingroup \smaller\smaller\smaller\begin{tabular}{@{}c@{}}%
30\\-5\\7
\end{tabular}\endgroup%
\kern3pt%
\begingroup \smaller\smaller\smaller\begin{tabular}{@{}c@{}}%
30\\-6\\4
\end{tabular}\endgroup%
\kern3pt%
\begingroup \smaller\smaller\smaller\begin{tabular}{@{}c@{}}%
90\\-19\\3
\end{tabular}\endgroup%
{$\left.\llap{\phantom{%
\begingroup \smaller\smaller\smaller\begin{tabular}{@{}c@{}}%
0\\0\\0
\end{tabular}\endgroup%
}}\!\right]$}%
}%
\ifdim\wd\matricesbox>\halfwidth\myboxwidth=\hsize\else\myboxwidth=\halfwidth\fi
\vbox{%
\ifdim\myboxwidth=\hsize
\setbox\onelinebox=\hbox{%
\vbox{\hbox{%
$\Pi_{20,7}$ spans $L_{16.13}$%
}\hbox{%
$\slashthree632223636\slashthree636322236\rtimes D_{2}$%
}%
}%
\hfill\copy\matricesbox
}%
\ifdim\wd\onelinebox>\myboxwidth
\hbox to \myboxwidth{%
$\Pi_{20,7}$ spans $L_{16.13}$%
\hfil
$\slashthree632223636\slashthree636322236\rtimes D_{2}$%
}%
\box\matricesbox
\else
\hbox to \myboxwidth{%
\unhbox\onelinebox
}%
\fi
\else
\hbox to \myboxwidth{%
$\Pi_{20,7}$ spans $L_{16.13}$%
\hfil}%
\hbox to \myboxwidth{%
$\slashthree632223636\slashthree636322236\rtimes D_{2}$%
\hfil}%
\box\matricesbox
\fi
}%
\hfill\discretionary{}{}{}%
\setbox\matricesbox=\hbox{%
{$\left[\!\llap{\phantom{%
\begingroup \smaller\smaller\smaller\begin{tabular}{@{}c@{}}%
\phantom{0}\\\phantom{0}\\\phantom{0}
\end{tabular}\endgroup%
}}\right.$}%
\begingroup \smaller\smaller\smaller\begin{tabular}{@{}c@{}}%
-3\\\phantom{0}\\\phantom{0}
\end{tabular}\endgroup%
\kern3pt%
\begingroup \smaller\smaller\smaller\begin{tabular}{@{}c@{}}%
\phantom{0}\\5/2\\\phantom{0}
\end{tabular}\endgroup%
\kern3pt%
\begingroup \smaller\smaller\smaller\begin{tabular}{@{}c@{}}%
\phantom{0}\\\phantom{0}\\15/2
\end{tabular}\endgroup%
{$\left.\llap{\phantom{%
\begingroup \smaller\smaller\smaller\begin{tabular}{@{}c@{}}%
\phantom{0}\\\phantom{0}\\\phantom{0}
\end{tabular}\endgroup%
}}\!\right]$}%
{$\left[\!\llap{\phantom{%
\begingroup \smaller\smaller\smaller\begin{tabular}{@{}c@{}}%
0\\0\\0
\end{tabular}\endgroup%
}}\right.$}%
\begingroup \smaller\smaller\smaller\begin{tabular}{@{}c@{}}%
10\\-11\\1
\end{tabular}\endgroup%
\kern3pt%
\begingroup \smaller\smaller\smaller\begin{tabular}{@{}c@{}}%
3\\-3\\1
\end{tabular}\endgroup%
\kern3pt%
\begingroup \smaller\smaller\smaller\begin{tabular}{@{}c@{}}%
10\\-7\\5
\end{tabular}\endgroup%
\kern3pt%
\begingroup \smaller\smaller\smaller\begin{tabular}{@{}c@{}}%
10\\-4\\6
\end{tabular}\endgroup%
\kern3pt%
\begingroup \smaller\smaller\smaller\begin{tabular}{@{}c@{}}%
3\\0\\2
\end{tabular}\endgroup%
\kern3pt%
\begingroup \smaller\smaller\smaller\begin{tabular}{@{}c@{}}%
10\\4\\6
\end{tabular}\endgroup%
\kern3pt%
\begingroup \smaller\smaller\smaller\begin{tabular}{@{}c@{}}%
10\\7\\5
\end{tabular}\endgroup%
\kern3pt%
\begingroup \smaller\smaller\smaller\begin{tabular}{@{}c@{}}%
30\\27\\11
\end{tabular}\endgroup%
\kern3pt%
\begingroup \smaller\smaller\smaller\begin{tabular}{@{}c@{}}%
30\\30\\8
\end{tabular}\endgroup%
\kern3pt%
\begingroup \smaller\smaller\smaller\begin{tabular}{@{}c@{}}%
10\\11\\1
\end{tabular}\endgroup%
{$\left.\llap{\phantom{%
\begingroup \smaller\smaller\smaller\begin{tabular}{@{}c@{}}%
0\\0\\0
\end{tabular}\endgroup%
}}\!\right]$}%
}%
\ifdim\wd\matricesbox>\halfwidth\myboxwidth=\hsize\else\myboxwidth=\halfwidth\fi
\vbox{%
\ifdim\myboxwidth=\hsize
\setbox\onelinebox=\hbox{%
\vbox{\hbox{%
$\Pi_{20,8}$ spans $L_{16.9}$%
}\hbox{%
$36322322\slashthree223223636\slashthree6\rtimes D_{2}$%
}%
}%
\hfill\copy\matricesbox
}%
\ifdim\wd\onelinebox>\myboxwidth
\hbox to \myboxwidth{%
$\Pi_{20,8}$ spans $L_{16.9}$%
\hfil
$36322322\slashthree223223636\slashthree6\rtimes D_{2}$%
}%
\box\matricesbox
\else
\hbox to \myboxwidth{%
\unhbox\onelinebox
}%
\fi
\else
\hbox to \myboxwidth{%
$\Pi_{20,8}$ spans $L_{16.9}$%
\hfil}%
\hbox to \myboxwidth{%
$36322322\slashthree223223636\slashthree6\rtimes D_{2}$%
\hfil}%
\box\matricesbox
\fi
}%
\hfill\discretionary{}{}{}%
\setbox\matricesbox=\hbox{%
{$\left[\!\llap{\phantom{%
\begingroup \smaller\smaller\smaller\begin{tabular}{@{}c@{}}%
\phantom{0}\\\phantom{0}\\\phantom{0}
\end{tabular}\endgroup%
}}\right.$}%
\begingroup \smaller\smaller\smaller\begin{tabular}{@{}c@{}}%
-1\\\phantom{0}\\\phantom{0}
\end{tabular}\endgroup%
\kern3pt%
\begingroup \smaller\smaller\smaller\begin{tabular}{@{}c@{}}%
\phantom{0}\\15/2\\\phantom{0}
\end{tabular}\endgroup%
\kern3pt%
\begingroup \smaller\smaller\smaller\begin{tabular}{@{}c@{}}%
\phantom{0}\\\phantom{0}\\45/2
\end{tabular}\endgroup%
{$\left.\llap{\phantom{%
\begingroup \smaller\smaller\smaller\begin{tabular}{@{}c@{}}%
\phantom{0}\\\phantom{0}\\\phantom{0}
\end{tabular}\endgroup%
}}\!\right]$}%
{$\left[\!\llap{\phantom{%
\begingroup \smaller\smaller\smaller\begin{tabular}{@{}c@{}}%
0\\0\\0
\end{tabular}\endgroup%
}}\right.$}%
\begingroup \smaller\smaller\smaller\begin{tabular}{@{}c@{}}%
30\\-11\\-1
\end{tabular}\endgroup%
\kern3pt%
\begingroup \smaller\smaller\smaller\begin{tabular}{@{}c@{}}%
9\\-3\\-1
\end{tabular}\endgroup%
\kern3pt%
\begingroup \smaller\smaller\smaller\begin{tabular}{@{}c@{}}%
30\\-7\\-5
\end{tabular}\endgroup%
\kern3pt%
\begingroup \smaller\smaller\smaller\begin{tabular}{@{}c@{}}%
30\\-4\\-6
\end{tabular}\endgroup%
\kern3pt%
\begingroup \smaller\smaller\smaller\begin{tabular}{@{}c@{}}%
90\\-3\\-19
\end{tabular}\endgroup%
\kern3pt%
\begingroup \smaller\smaller\smaller\begin{tabular}{@{}c@{}}%
90\\3\\-19
\end{tabular}\endgroup%
\kern3pt%
\begingroup \smaller\smaller\smaller\begin{tabular}{@{}c@{}}%
5\\1\\-1
\end{tabular}\endgroup%
\kern3pt%
\begingroup \smaller\smaller\smaller\begin{tabular}{@{}c@{}}%
90\\27\\-11
\end{tabular}\endgroup%
\kern3pt%
\begingroup \smaller\smaller\smaller\begin{tabular}{@{}c@{}}%
90\\30\\-8
\end{tabular}\endgroup%
\kern3pt%
\begingroup \smaller\smaller\smaller\begin{tabular}{@{}c@{}}%
30\\11\\-1
\end{tabular}\endgroup%
{$\left.\llap{\phantom{%
\begingroup \smaller\smaller\smaller\begin{tabular}{@{}c@{}}%
0\\0\\0
\end{tabular}\endgroup%
}}\!\right]$}%
}%
\ifdim\wd\matricesbox>\halfwidth\myboxwidth=\hsize\else\myboxwidth=\halfwidth\fi
\vbox{%
\ifdim\myboxwidth=\hsize
\setbox\onelinebox=\hbox{%
\vbox{\hbox{%
$\Pi_{20,9}$ spans $L_{16.13}$%
}\hbox{%
$36322\slashthree223632236\slashthree6322\rtimes D_{2}$%
}%
}%
\hfill\copy\matricesbox
}%
\ifdim\wd\onelinebox>\myboxwidth
\hbox to \myboxwidth{%
$\Pi_{20,9}$ spans $L_{16.13}$%
\hfil
$36322\slashthree223632236\slashthree6322\rtimes D_{2}$%
}%
\box\matricesbox
\else
\hbox to \myboxwidth{%
\unhbox\onelinebox
}%
\fi
\else
\hbox to \myboxwidth{%
$\Pi_{20,9}$ spans $L_{16.13}$%
\hfil}%
\hbox to \myboxwidth{%
$36322\slashthree223632236\slashthree6322\rtimes D_{2}$%
\hfil}%
\box\matricesbox
\fi
}%
\hfill\discretionary{}{}{}%
\setbox\matricesbox=\hbox{%
{$\left[\!\llap{\phantom{%
\begingroup \smaller\smaller\smaller\begin{tabular}{@{}c@{}}%
\phantom{0}\\\phantom{0}\\\phantom{0}
\end{tabular}\endgroup%
}}\right.$}%
\begingroup \smaller\smaller\smaller\begin{tabular}{@{}c@{}}%
-1\\\phantom{0}\\\phantom{0}
\end{tabular}\endgroup%
\kern3pt%
\begingroup \smaller\smaller\smaller\begin{tabular}{@{}c@{}}%
\phantom{0}\\45/2\\\phantom{0}
\end{tabular}\endgroup%
\kern3pt%
\begingroup \smaller\smaller\smaller\begin{tabular}{@{}c@{}}%
\phantom{0}\\\phantom{0}\\15/2
\end{tabular}\endgroup%
{$\left.\llap{\phantom{%
\begingroup \smaller\smaller\smaller\begin{tabular}{@{}c@{}}%
\phantom{0}\\\phantom{0}\\\phantom{0}
\end{tabular}\endgroup%
}}\!\right]$}%
{$\left[\!\llap{\phantom{%
\begingroup \smaller\smaller\smaller\begin{tabular}{@{}c@{}}%
0\\0\\0
\end{tabular}\endgroup%
}}\right.$}%
\begingroup \smaller\smaller\smaller\begin{tabular}{@{}c@{}}%
90\\19\\3
\end{tabular}\endgroup%
\kern3pt%
\begingroup \smaller\smaller\smaller\begin{tabular}{@{}c@{}}%
30\\6\\4
\end{tabular}\endgroup%
\kern3pt%
\begingroup \smaller\smaller\smaller\begin{tabular}{@{}c@{}}%
30\\5\\7
\end{tabular}\endgroup%
\kern3pt%
\begingroup \smaller\smaller\smaller\begin{tabular}{@{}c@{}}%
9\\1\\3
\end{tabular}\endgroup%
\kern3pt%
\begingroup \smaller\smaller\smaller\begin{tabular}{@{}c@{}}%
30\\1\\11
\end{tabular}\endgroup%
\kern3pt%
\begingroup \smaller\smaller\smaller\begin{tabular}{@{}c@{}}%
30\\-1\\11
\end{tabular}\endgroup%
\kern3pt%
\begingroup \smaller\smaller\smaller\begin{tabular}{@{}c@{}}%
90\\-8\\30
\end{tabular}\endgroup%
\kern3pt%
\begingroup \smaller\smaller\smaller\begin{tabular}{@{}c@{}}%
90\\-11\\27
\end{tabular}\endgroup%
\kern3pt%
\begingroup \smaller\smaller\smaller\begin{tabular}{@{}c@{}}%
5\\-1\\1
\end{tabular}\endgroup%
\kern3pt%
\begingroup \smaller\smaller\smaller\begin{tabular}{@{}c@{}}%
90\\-19\\3
\end{tabular}\endgroup%
{$\left.\llap{\phantom{%
\begingroup \smaller\smaller\smaller\begin{tabular}{@{}c@{}}%
0\\0\\0
\end{tabular}\endgroup%
}}\!\right]$}%
}%
\ifdim\wd\matricesbox>\halfwidth\myboxwidth=\hsize\else\myboxwidth=\halfwidth\fi
\vbox{%
\ifdim\myboxwidth=\hsize
\setbox\onelinebox=\hbox{%
\vbox{\hbox{%
$\Pi_{20,10}$ spans $L_{16.13}$%
}\hbox{%
$\slashthree632236322\slashthree223632236\rtimes D_{2}$%
}%
}%
\hfill\copy\matricesbox
}%
\ifdim\wd\onelinebox>\myboxwidth
\hbox to \myboxwidth{%
$\Pi_{20,10}$ spans $L_{16.13}$%
\hfil
$\slashthree632236322\slashthree223632236\rtimes D_{2}$%
}%
\box\matricesbox
\else
\hbox to \myboxwidth{%
\unhbox\onelinebox
}%
\fi
\else
\hbox to \myboxwidth{%
$\Pi_{20,10}$ spans $L_{16.13}$%
\hfil}%
\hbox to \myboxwidth{%
$\slashthree632236322\slashthree223632236\rtimes D_{2}$%
\hfil}%
\box\matricesbox
\fi
}%
\hfill\discretionary{}{}{}%
\setbox\matricesbox=\hbox{%
{$\left[\!\llap{\phantom{%
\begingroup \smaller\smaller\smaller
\endgroup%
}}\!\right]$}%
}%
\ifdim\wd\matricesbox>\halfwidth\myboxwidth=\hsize\else\myboxwidth=\halfwidth\fi
\vbox{%
\ifdim\myboxwidth=\hsize
\setbox\onelinebox=\hbox{%
\vbox{\hbox{%
$\Pi_{20,11}$ spans $L_{16.13}$%
}\hbox{%
$3632|2363636322|223636\rtimes D_{2}$%
}%
}%
\hfill\copy\matricesbox
}%
\ifdim\wd\onelinebox>\myboxwidth
\hbox to \myboxwidth{%
$\Pi_{20,11}$ spans $L_{16.13}$%
\hfil
$3632|2363636322|223636\rtimes D_{2}$%
}%
\box\matricesbox
\else
\hbox to \myboxwidth{%
\unhbox\onelinebox
}%
\fi
\else
\hbox to \myboxwidth{%
$\Pi_{20,11}$ spans $L_{16.13}$%
\hfil}%
\hbox to \myboxwidth{%
$3632|2363636322|223636\rtimes D_{2}$%
\hfil}%
\box\matricesbox
\fi
}%
\hfill\discretionary{}{}{}%
\setbox\matricesbox=\hbox{%
{$\left[\!\llap{\phantom{%
\begingroup \smaller\smaller\smaller\begin{tabular}{@{}c@{}}%
\phantom{0}\\\phantom{0}\\\phantom{0}
\end{tabular}\endgroup%
}}\right.$}%
\begingroup \smaller\smaller\smaller\begin{tabular}{@{}c@{}}%
-5\\\phantom{0}\\\phantom{0}
\end{tabular}\endgroup%
\kern3pt%
\begingroup \smaller\smaller\smaller\begin{tabular}{@{}c@{}}%
\phantom{0}\\9/2\\\phantom{0}
\end{tabular}\endgroup%
\kern3pt%
\begingroup \smaller\smaller\smaller\begin{tabular}{@{}c@{}}%
\phantom{0}\\\phantom{0}\\3/2
\end{tabular}\endgroup%
{$\left.\llap{\phantom{%
\begingroup \smaller\smaller\smaller\begin{tabular}{@{}c@{}}%
\phantom{0}\\\phantom{0}\\\phantom{0}
\end{tabular}\endgroup%
}}\!\right]$}%
{$\left[\!\llap{\phantom{%
\begingroup \smaller\smaller\smaller\begin{tabular}{@{}c@{}}%
0\\0\\0
\end{tabular}\endgroup%
}}\right.$}%
\begingroup \smaller\smaller\smaller\begin{tabular}{@{}c@{}}%
18\\19\\3
\end{tabular}\endgroup%
\kern3pt%
\begingroup \smaller\smaller\smaller\begin{tabular}{@{}c@{}}%
6\\6\\4
\end{tabular}\endgroup%
\kern3pt%
\begingroup \smaller\smaller\smaller\begin{tabular}{@{}c@{}}%
6\\5\\7
\end{tabular}\endgroup%
\kern3pt%
\begingroup \smaller\smaller\smaller\begin{tabular}{@{}c@{}}%
18\\11\\27
\end{tabular}\endgroup%
\kern3pt%
\begingroup \smaller\smaller\smaller\begin{tabular}{@{}c@{}}%
18\\8\\30
\end{tabular}\endgroup%
\kern3pt%
\begingroup \smaller\smaller\smaller\begin{tabular}{@{}c@{}}%
1\\0\\2
\end{tabular}\endgroup%
\kern3pt%
\begingroup \smaller\smaller\smaller\begin{tabular}{@{}c@{}}%
18\\-8\\30
\end{tabular}\endgroup%
\kern3pt%
\begingroup \smaller\smaller\smaller\begin{tabular}{@{}c@{}}%
18\\-11\\27
\end{tabular}\endgroup%
\kern3pt%
\begingroup \smaller\smaller\smaller\begin{tabular}{@{}c@{}}%
1\\-1\\1
\end{tabular}\endgroup%
\kern3pt%
\begingroup \smaller\smaller\smaller\begin{tabular}{@{}c@{}}%
18\\-19\\3
\end{tabular}\endgroup%
{$\left.\llap{\phantom{%
\begingroup \smaller\smaller\smaller\begin{tabular}{@{}c@{}}%
0\\0\\0
\end{tabular}\endgroup%
}}\!\right]$}%
}%
\ifdim\wd\matricesbox>\halfwidth\myboxwidth=\hsize\else\myboxwidth=\halfwidth\fi
\vbox{%
\ifdim\myboxwidth=\hsize
\setbox\onelinebox=\hbox{%
\vbox{\hbox{%
$\Pi_{20,12}$ spans $L_{16.7}$%
}\hbox{%
$\slashthree636322322\slashthree223223636\rtimes D_{2}$%
}%
}%
\hfill\copy\matricesbox
}%
\ifdim\wd\onelinebox>\myboxwidth
\hbox to \myboxwidth{%
$\Pi_{20,12}$ spans $L_{16.7}$%
\hfil
$\slashthree636322322\slashthree223223636\rtimes D_{2}$%
}%
\box\matricesbox
\else
\hbox to \myboxwidth{%
\unhbox\onelinebox
}%
\fi
\else
\hbox to \myboxwidth{%
$\Pi_{20,12}$ spans $L_{16.7}$%
\hfil}%
\hbox to \myboxwidth{%
$\slashthree636322322\slashthree223223636\rtimes D_{2}$%
\hfil}%
\box\matricesbox
\fi
}%
\hfill\discretionary{}{}{}%
\setbox\matricesbox=\hbox{%
{$\left[\!\llap{\phantom{%
\begingroup \smaller\smaller\smaller
\endgroup%
}}\!\right]$}%
}%
\ifdim\wd\matricesbox>\halfwidth\myboxwidth=\hsize\else\myboxwidth=\halfwidth\fi
\vbox{%
\ifdim\myboxwidth=\hsize
\setbox\onelinebox=\hbox{%
\vbox{\hbox{%
$\Pi_{20,13}$ spans $L_{16.7}$%
}\hbox{%
$363632232|2322363632|2\rtimes D_{2}$%
}%
}%
\hfill\copy\matricesbox
}%
\ifdim\wd\onelinebox>\myboxwidth
\hbox to \myboxwidth{%
$\Pi_{20,13}$ spans $L_{16.7}$%
\hfil
$363632232|2322363632|2\rtimes D_{2}$%
}%
\box\matricesbox
\else
\hbox to \myboxwidth{%
\unhbox\onelinebox
}%
\fi
\else
\hbox to \myboxwidth{%
$\Pi_{20,13}$ spans $L_{16.7}$%
\hfil}%
\hbox to \myboxwidth{%
$363632232|2322363632|2\rtimes D_{2}$%
\hfil}%
\box\matricesbox
\fi
}%
\hfill\discretionary{}{}{}%
\setbox\matricesbox=\hbox{%
{$\left[\!\llap{\phantom{%
\begingroup \smaller\smaller\smaller
\endgroup%
}}\!\right]$}%
}%
\ifdim\wd\matricesbox>\halfwidth\myboxwidth=\hsize\else\myboxwidth=\halfwidth\fi
\vbox{%
\ifdim\myboxwidth=\hsize
\setbox\onelinebox=\hbox{%
\vbox{\hbox{%
$\Pi_{20,14}$ spans $L_{251.3}$%
}\hbox{%
$32232|2322322262|26222\rtimes D_{2}$%
}%
}%
\hfill\copy\matricesbox
}%
\ifdim\wd\onelinebox>\myboxwidth
\hbox to \myboxwidth{%
$\Pi_{20,14}$ spans $L_{251.3}$%
\hfil
$32232|2322322262|26222\rtimes D_{2}$%
}%
\box\matricesbox
\else
\hbox to \myboxwidth{%
\unhbox\onelinebox
}%
\fi
\else
\hbox to \myboxwidth{%
$\Pi_{20,14}$ spans $L_{251.3}$%
\hfil}%
\hbox to \myboxwidth{%
$32232|2322322262|26222\rtimes D_{2}$%
\hfil}%
\box\matricesbox
\fi
}%
\hfill\discretionary{}{}{}%
\setbox\matricesbox=\hbox{%
{$\left[\!\llap{\phantom{%
\begingroup \smaller\smaller\smaller
\endgroup%
}}\!\right]$}%
}%
\ifdim\wd\matricesbox>\halfwidth\myboxwidth=\hsize\else\myboxwidth=\halfwidth\fi
\vbox{%
\ifdim\myboxwidth=\hsize
\setbox\onelinebox=\hbox{%
\vbox{\hbox{%
$\Pi_{20,15}$ spans $L_{251.3}$%
}\hbox{%
$32232|2322322622|22622\rtimes D_{2}$%
}%
}%
\hfill\copy\matricesbox
}%
\ifdim\wd\onelinebox>\myboxwidth
\hbox to \myboxwidth{%
$\Pi_{20,15}$ spans $L_{251.3}$%
\hfil
$32232|2322322622|22622\rtimes D_{2}$%
}%
\box\matricesbox
\else
\hbox to \myboxwidth{%
\unhbox\onelinebox
}%
\fi
\else
\hbox to \myboxwidth{%
$\Pi_{20,15}$ spans $L_{251.3}$%
\hfil}%
\hbox to \myboxwidth{%
$32232|2322322622|22622\rtimes D_{2}$%
\hfil}%
\box\matricesbox
\fi
}%
\hfill\discretionary{}{}{}%
\setbox\matricesbox=\hbox{%
{$\left[\!\llap{\phantom{%
\begingroup \smaller\smaller\smaller\begin{tabular}{@{}c@{}}%
\phantom{0}\\\phantom{0}\\\phantom{0}
\end{tabular}\endgroup%
}}\right.$}%
\begingroup \smaller\smaller\smaller\begin{tabular}{@{}c@{}}%
-5\\\phantom{0}\\\phantom{0}
\end{tabular}\endgroup%
\kern3pt%
\begingroup \smaller\smaller\smaller\begin{tabular}{@{}c@{}}%
\phantom{0}\\9\\\phantom{0}
\end{tabular}\endgroup%
\kern3pt%
\begingroup \smaller\smaller\smaller\begin{tabular}{@{}c@{}}%
\phantom{0}\\\phantom{0}\\3
\end{tabular}\endgroup%
{$\left.\llap{\phantom{%
\begingroup \smaller\smaller\smaller\begin{tabular}{@{}c@{}}%
\phantom{0}\\\phantom{0}\\\phantom{0}
\end{tabular}\endgroup%
}}\!\right]$}%
{$\left[\!\llap{\phantom{%
\begingroup \smaller\smaller\smaller\begin{tabular}{@{}c@{}}%
0\\0\\0
\end{tabular}\endgroup%
}}\right.$}%
\begingroup \smaller\smaller\smaller\begin{tabular}{@{}c@{}}%
4\\-3\\1
\end{tabular}\endgroup%
\kern3pt%
\begingroup \smaller\smaller\smaller\begin{tabular}{@{}c@{}}%
3\\-2\\2
\end{tabular}\endgroup%
\kern3pt%
\begingroup \smaller\smaller\smaller\begin{tabular}{@{}c@{}}%
4\\-2\\4
\end{tabular}\endgroup%
\kern3pt%
\begingroup \smaller\smaller\smaller\begin{tabular}{@{}c@{}}%
4\\-1\\5
\end{tabular}\endgroup%
\kern3pt%
\begingroup \smaller\smaller\smaller\begin{tabular}{@{}c@{}}%
3\\0\\4
\end{tabular}\endgroup%
\kern3pt%
\begingroup \smaller\smaller\smaller\begin{tabular}{@{}c@{}}%
4\\1\\5
\end{tabular}\endgroup%
\kern3pt%
\begingroup \smaller\smaller\smaller\begin{tabular}{@{}c@{}}%
12\\5\\13
\end{tabular}\endgroup%
\kern3pt%
\begingroup \smaller\smaller\smaller\begin{tabular}{@{}c@{}}%
36\\19\\33
\end{tabular}\endgroup%
\kern3pt%
\begingroup \smaller\smaller\smaller\begin{tabular}{@{}c@{}}%
3\\2\\2
\end{tabular}\endgroup%
\kern3pt%
\begingroup \smaller\smaller\smaller\begin{tabular}{@{}c@{}}%
4\\3\\1
\end{tabular}\endgroup%
{$\left.\llap{\phantom{%
\begingroup \smaller\smaller\smaller\begin{tabular}{@{}c@{}}%
0\\0\\0
\end{tabular}\endgroup%
}}\!\right]$}%
}%
\ifdim\wd\matricesbox>\halfwidth\myboxwidth=\hsize\else\myboxwidth=\halfwidth\fi
\vbox{%
\ifdim\myboxwidth=\hsize
\setbox\onelinebox=\hbox{%
\vbox{\hbox{%
$\Pi_{20,16}$ spans $L_{251.3}$%
}\hbox{%
$322\slashthree223222622\slashthree226222\rtimes D_{2}$%
}%
}%
\hfill\copy\matricesbox
}%
\ifdim\wd\onelinebox>\myboxwidth
\hbox to \myboxwidth{%
$\Pi_{20,16}$ spans $L_{251.3}$%
\hfil
$322\slashthree223222622\slashthree226222\rtimes D_{2}$%
}%
\box\matricesbox
\else
\hbox to \myboxwidth{%
\unhbox\onelinebox
}%
\fi
\else
\hbox to \myboxwidth{%
$\Pi_{20,16}$ spans $L_{251.3}$%
\hfil}%
\hbox to \myboxwidth{%
$322\slashthree223222622\slashthree226222\rtimes D_{2}$%
\hfil}%
\box\matricesbox
\fi
}%
\hfill\discretionary{}{}{}%
\setbox\matricesbox=\hbox{%
{$\left[\!\llap{\phantom{%
\begingroup \smaller\smaller\smaller\begin{tabular}{@{}c@{}}%
\phantom{0}\\\phantom{0}\\\phantom{0}
\end{tabular}\endgroup%
}}\right.$}%
\begingroup \smaller\smaller\smaller\begin{tabular}{@{}c@{}}%
-5\\\phantom{0}\\\phantom{0}
\end{tabular}\endgroup%
\kern3pt%
\begingroup \smaller\smaller\smaller\begin{tabular}{@{}c@{}}%
\phantom{0}\\9\\\phantom{0}
\end{tabular}\endgroup%
\kern3pt%
\begingroup \smaller\smaller\smaller\begin{tabular}{@{}c@{}}%
\phantom{0}\\\phantom{0}\\3
\end{tabular}\endgroup%
{$\left.\llap{\phantom{%
\begingroup \smaller\smaller\smaller\begin{tabular}{@{}c@{}}%
\phantom{0}\\\phantom{0}\\\phantom{0}
\end{tabular}\endgroup%
}}\!\right]$}%
{$\left[\!\llap{\phantom{%
\begingroup \smaller\smaller\smaller\begin{tabular}{@{}c@{}}%
0\\0\\0
\end{tabular}\endgroup%
}}\right.$}%
\begingroup \smaller\smaller\smaller\begin{tabular}{@{}c@{}}%
4\\-3\\1
\end{tabular}\endgroup%
\kern3pt%
\begingroup \smaller\smaller\smaller\begin{tabular}{@{}c@{}}%
3\\-2\\2
\end{tabular}\endgroup%
\kern3pt%
\begingroup \smaller\smaller\smaller\begin{tabular}{@{}c@{}}%
4\\-2\\4
\end{tabular}\endgroup%
\kern3pt%
\begingroup \smaller\smaller\smaller\begin{tabular}{@{}c@{}}%
4\\-1\\5
\end{tabular}\endgroup%
\kern3pt%
\begingroup \smaller\smaller\smaller\begin{tabular}{@{}c@{}}%
3\\0\\4
\end{tabular}\endgroup%
\kern3pt%
\begingroup \smaller\smaller\smaller\begin{tabular}{@{}c@{}}%
36\\7\\45
\end{tabular}\endgroup%
\kern3pt%
\begingroup \smaller\smaller\smaller\begin{tabular}{@{}c@{}}%
12\\4\\14
\end{tabular}\endgroup%
\kern3pt%
\begingroup \smaller\smaller\smaller\begin{tabular}{@{}c@{}}%
4\\2\\4
\end{tabular}\endgroup%
\kern3pt%
\begingroup \smaller\smaller\smaller\begin{tabular}{@{}c@{}}%
3\\2\\2
\end{tabular}\endgroup%
\kern3pt%
\begingroup \smaller\smaller\smaller\begin{tabular}{@{}c@{}}%
4\\3\\1
\end{tabular}\endgroup%
{$\left.\llap{\phantom{%
\begingroup \smaller\smaller\smaller\begin{tabular}{@{}c@{}}%
0\\0\\0
\end{tabular}\endgroup%
}}\!\right]$}%
}%
\ifdim\wd\matricesbox>\halfwidth\myboxwidth=\hsize\else\myboxwidth=\halfwidth\fi
\vbox{%
\ifdim\myboxwidth=\hsize
\setbox\onelinebox=\hbox{%
\vbox{\hbox{%
$\Pi_{20,17}$ spans $L_{251.3}$%
}\hbox{%
$322\slashthree223226222\slashthree222622\rtimes D_{2}$%
}%
}%
\hfill\copy\matricesbox
}%
\ifdim\wd\onelinebox>\myboxwidth
\hbox to \myboxwidth{%
$\Pi_{20,17}$ spans $L_{251.3}$%
\hfil
$322\slashthree223226222\slashthree222622\rtimes D_{2}$%
}%
\box\matricesbox
\else
\hbox to \myboxwidth{%
\unhbox\onelinebox
}%
\fi
\else
\hbox to \myboxwidth{%
$\Pi_{20,17}$ spans $L_{251.3}$%
\hfil}%
\hbox to \myboxwidth{%
$322\slashthree223226222\slashthree222622\rtimes D_{2}$%
\hfil}%
\box\matricesbox
\fi
}%
\hfill\discretionary{}{}{}%
\setbox\matricesbox=\hbox{%
{$\left[\!\llap{\phantom{%
\begingroup \smaller\smaller\smaller
\endgroup%
}}\!\right]$}%
}%
\ifdim\wd\matricesbox>\halfwidth\myboxwidth=\hsize\else\myboxwidth=\halfwidth\fi
\vbox{%
\ifdim\myboxwidth=\hsize
\setbox\onelinebox=\hbox{%
\vbox{\hbox{%
$\Pi_{20,18}$ spans $L_{251.3}$%
}\hbox{%
$32|2322262232|23226222\rtimes D_{2}$%
}%
}%
\hfill\copy\matricesbox
}%
\ifdim\wd\onelinebox>\myboxwidth
\hbox to \myboxwidth{%
$\Pi_{20,18}$ spans $L_{251.3}$%
\hfil
$32|2322262232|23226222\rtimes D_{2}$%
}%
\box\matricesbox
\else
\hbox to \myboxwidth{%
\unhbox\onelinebox
}%
\fi
\else
\hbox to \myboxwidth{%
$\Pi_{20,18}$ spans $L_{251.3}$%
\hfil}%
\hbox to \myboxwidth{%
$32|2322262232|23226222\rtimes D_{2}$%
\hfil}%
\box\matricesbox
\fi
}%
\hfill\discretionary{}{}{}%
\setbox\matricesbox=\hbox{%
{$\left[\!\llap{\phantom{%
\begingroup \smaller\smaller\smaller
\endgroup%
}}\!\right]$}%
}%
\ifdim\wd\matricesbox>\halfwidth\myboxwidth=\hsize\else\myboxwidth=\halfwidth\fi
\vbox{%
\ifdim\myboxwidth=\hsize
\setbox\onelinebox=\hbox{%
\vbox{\hbox{%
$\Pi_{20,19}$ spans $L_{16.13}$%
}\hbox{%
$36322|2236363632|23636\rtimes D_{2}$%
}%
}%
\hfill\copy\matricesbox
}%
\ifdim\wd\onelinebox>\myboxwidth
\hbox to \myboxwidth{%
$\Pi_{20,19}$ spans $L_{16.13}$%
\hfil
$36322|2236363632|23636\rtimes D_{2}$%
}%
\box\matricesbox
\else
\hbox to \myboxwidth{%
\unhbox\onelinebox
}%
\fi
\else
\hbox to \myboxwidth{%
$\Pi_{20,19}$ spans $L_{16.13}$%
\hfil}%
\hbox to \myboxwidth{%
$36322|2236363632|23636\rtimes D_{2}$%
\hfil}%
\box\matricesbox
\fi
}%
\hfill\discretionary{}{}{}%
\setbox\matricesbox=\hbox{%
{$\left[\!\llap{\phantom{%
\begingroup \smaller\smaller\smaller
\endgroup%
}}\!\right]$}%
}%
\ifdim\wd\matricesbox>\halfwidth\myboxwidth=\hsize\else\myboxwidth=\halfwidth\fi
\vbox{%
\ifdim\myboxwidth=\hsize
\setbox\onelinebox=\hbox{%
\vbox{\hbox{%
$\Pi_{20,20}$ spans $L_{16.9}$%
}\hbox{%
$3632232|2322363632|236\rtimes D_{2}$%
}%
}%
\hfill\copy\matricesbox
}%
\ifdim\wd\onelinebox>\myboxwidth
\hbox to \myboxwidth{%
$\Pi_{20,20}$ spans $L_{16.9}$%
\hfil
$3632232|2322363632|236\rtimes D_{2}$%
}%
\box\matricesbox
\else
\hbox to \myboxwidth{%
\unhbox\onelinebox
}%
\fi
\else
\hbox to \myboxwidth{%
$\Pi_{20,20}$ spans $L_{16.9}$%
\hfil}%
\hbox to \myboxwidth{%
$3632232|2322363632|236\rtimes D_{2}$%
\hfil}%
\box\matricesbox
\fi
}%
\hfill\discretionary{}{}{}%
\setbox\matricesbox=\hbox{%
{$\left[\!\llap{\phantom{%
\begingroup \smaller\smaller\smaller\begin{tabular}{@{}c@{}}%
\phantom{0}\\\phantom{0}\\\phantom{0}
\end{tabular}\endgroup%
}}\right.$}%
\begingroup \smaller\smaller\smaller\begin{tabular}{@{}c@{}}%
-1\\\phantom{0}\\\phantom{0}
\end{tabular}\endgroup%
\kern3pt%
\begingroup \smaller\smaller\smaller\begin{tabular}{@{}c@{}}%
\phantom{0}\\15/2\\\phantom{0}
\end{tabular}\endgroup%
\kern3pt%
\begingroup \smaller\smaller\smaller\begin{tabular}{@{}c@{}}%
\phantom{0}\\\phantom{0}\\45/2
\end{tabular}\endgroup%
{$\left.\llap{\phantom{%
\begingroup \smaller\smaller\smaller\begin{tabular}{@{}c@{}}%
\phantom{0}\\\phantom{0}\\\phantom{0}
\end{tabular}\endgroup%
}}\!\right]$}%
{$\left[\!\llap{\phantom{%
\begingroup \smaller\smaller\smaller\begin{tabular}{@{}c@{}}%
0\\0\\0
\end{tabular}\endgroup%
}}\right.$}%
\begingroup \smaller\smaller\smaller\begin{tabular}{@{}c@{}}%
30\\11\\-1
\end{tabular}\endgroup%
\kern3pt%
\begingroup \smaller\smaller\smaller\begin{tabular}{@{}c@{}}%
9\\3\\-1
\end{tabular}\endgroup%
\kern3pt%
\begingroup \smaller\smaller\smaller\begin{tabular}{@{}c@{}}%
5\\1\\-1
\end{tabular}\endgroup%
\kern3pt%
\begingroup \smaller\smaller\smaller\begin{tabular}{@{}c@{}}%
90\\3\\-19
\end{tabular}\endgroup%
\kern3pt%
\begingroup \smaller\smaller\smaller\begin{tabular}{@{}c@{}}%
90\\-3\\-19
\end{tabular}\endgroup%
\kern3pt%
\begingroup \smaller\smaller\smaller\begin{tabular}{@{}c@{}}%
30\\-4\\-6
\end{tabular}\endgroup%
\kern3pt%
\begingroup \smaller\smaller\smaller\begin{tabular}{@{}c@{}}%
30\\-7\\-5
\end{tabular}\endgroup%
\kern3pt%
\begingroup \smaller\smaller\smaller\begin{tabular}{@{}c@{}}%
90\\-27\\-11
\end{tabular}\endgroup%
\kern3pt%
\begingroup \smaller\smaller\smaller\begin{tabular}{@{}c@{}}%
90\\-30\\-8
\end{tabular}\endgroup%
\kern3pt%
\begingroup \smaller\smaller\smaller\begin{tabular}{@{}c@{}}%
30\\-11\\-1
\end{tabular}\endgroup%
{$\left.\llap{\phantom{%
\begingroup \smaller\smaller\smaller\begin{tabular}{@{}c@{}}%
0\\0\\0
\end{tabular}\endgroup%
}}\!\right]$}%
}%
\ifdim\wd\matricesbox>\halfwidth\myboxwidth=\hsize\else\myboxwidth=\halfwidth\fi
\vbox{%
\ifdim\myboxwidth=\hsize
\setbox\onelinebox=\hbox{%
\vbox{\hbox{%
$\Pi_{20,21}$ spans $L_{16.13}$%
}\hbox{%
$36363222\slashthree222363636\slashthree6\rtimes D_{2}$%
}%
}%
\hfill\copy\matricesbox
}%
\ifdim\wd\onelinebox>\myboxwidth
\hbox to \myboxwidth{%
$\Pi_{20,21}$ spans $L_{16.13}$%
\hfil
$36363222\slashthree222363636\slashthree6\rtimes D_{2}$%
}%
\box\matricesbox
\else
\hbox to \myboxwidth{%
\unhbox\onelinebox
}%
\fi
\else
\hbox to \myboxwidth{%
$\Pi_{20,21}$ spans $L_{16.13}$%
\hfil}%
\hbox to \myboxwidth{%
$36363222\slashthree222363636\slashthree6\rtimes D_{2}$%
\hfil}%
\box\matricesbox
\fi
}%
\hfill\discretionary{}{}{}%
\setbox\matricesbox=\hbox{%
{$\left[\!\llap{\phantom{%
\begingroup \smaller\smaller\smaller\begin{tabular}{@{}c@{}}%
\phantom{0}\\\phantom{0}\\\phantom{0}
\end{tabular}\endgroup%
}}\right.$}%
\begingroup \smaller\smaller\smaller\begin{tabular}{@{}c@{}}%
-1\\\phantom{0}\\\phantom{0}
\end{tabular}\endgroup%
\kern3pt%
\begingroup \smaller\smaller\smaller\begin{tabular}{@{}c@{}}%
\phantom{0}\\30\\-15
\end{tabular}\endgroup%
\kern3pt%
\begingroup \smaller\smaller\smaller\begin{tabular}{@{}c@{}}%
\phantom{0}\\-15\\30
\end{tabular}\endgroup%
{$\left.\llap{\phantom{%
\begingroup \smaller\smaller\smaller\begin{tabular}{@{}c@{}}%
\phantom{0}\\\phantom{0}\\\phantom{0}
\end{tabular}\endgroup%
}}\!\right]$}%
{$\left[\!\llap{\phantom{%
\begingroup \smaller\smaller\smaller\begin{tabular}{@{}c@{}}%
0\\0\\0
\end{tabular}\endgroup%
}}\right.$}%
\begingroup \smaller\smaller\smaller\begin{tabular}{@{}c@{}}%
90\\8\\19
\end{tabular}\endgroup%
\kern3pt%
\begingroup \smaller\smaller\smaller\begin{tabular}{@{}c@{}}%
90\\11\\19
\end{tabular}\endgroup%
\kern3pt%
\begingroup \smaller\smaller\smaller\begin{tabular}{@{}c@{}}%
30\\5\\6
\end{tabular}\endgroup%
\kern3pt%
\begingroup \smaller\smaller\smaller\begin{tabular}{@{}c@{}}%
30\\6\\5
\end{tabular}\endgroup%
\kern3pt%
\begingroup \smaller\smaller\smaller\begin{tabular}{@{}c@{}}%
9\\2\\1
\end{tabular}\endgroup%
\kern3pt%
\begingroup \smaller\smaller\smaller\begin{tabular}{@{}c@{}}%
5\\1\\0
\end{tabular}\endgroup%
\kern3pt%
\begingroup \smaller\smaller\smaller\begin{tabular}{@{}c@{}}%
90\\11\\-8
\end{tabular}\endgroup%
\kern3pt%
\begingroup \smaller\smaller\smaller\begin{tabular}{@{}c@{}}%
90\\8\\-11
\end{tabular}\endgroup%
\kern3pt%
\begingroup \smaller\smaller\smaller\begin{tabular}{@{}c@{}}%
30\\1\\-5
\end{tabular}\endgroup%
\kern3pt%
\begingroup \smaller\smaller\smaller\begin{tabular}{@{}c@{}}%
30\\-1\\-6
\end{tabular}\endgroup%
{$\left.\llap{\phantom{%
\begingroup \smaller\smaller\smaller\begin{tabular}{@{}c@{}}%
0\\0\\0
\end{tabular}\endgroup%
}}\!\right]$}%
}%
\ifdim\wd\matricesbox>\halfwidth\myboxwidth=\hsize\else\myboxwidth=\halfwidth\fi
\vbox{%
\ifdim\myboxwidth=\hsize
\setbox\onelinebox=\hbox{%
\vbox{\hbox{%
$\Pi_{20,22}$ spans $L_{16.13}$%
}\hbox{%
$36322236363632223636\rtimes C_{2}$%
}%
}%
\hfill\copy\matricesbox
}%
\ifdim\wd\onelinebox>\myboxwidth
\hbox to \myboxwidth{%
$\Pi_{20,22}$ spans $L_{16.13}$%
\hfil
$36322236363632223636\rtimes C_{2}$%
}%
\box\matricesbox
\else
\hbox to \myboxwidth{%
\unhbox\onelinebox
}%
\fi
\else
\hbox to \myboxwidth{%
$\Pi_{20,22}$ spans $L_{16.13}$%
\hfil}%
\hbox to \myboxwidth{%
$36322236363632223636\rtimes C_{2}$%
\hfil}%
\box\matricesbox
\fi
}%
\hfill\discretionary{}{}{}%
\setbox\matricesbox=\hbox{%
{$\left[\!\llap{\phantom{%
\begingroup \smaller\smaller\smaller\begin{tabular}{@{}c@{}}%
\phantom{0}\\\phantom{0}\\\phantom{0}
\end{tabular}\endgroup%
}}\right.$}%
\begingroup \smaller\smaller\smaller\begin{tabular}{@{}c@{}}%
-5\\\phantom{0}\\\phantom{0}
\end{tabular}\endgroup%
\kern3pt%
\begingroup \smaller\smaller\smaller\begin{tabular}{@{}c@{}}%
\phantom{0}\\12\\-6
\end{tabular}\endgroup%
\kern3pt%
\begingroup \smaller\smaller\smaller\begin{tabular}{@{}c@{}}%
\phantom{0}\\-6\\12
\end{tabular}\endgroup%
{$\left.\llap{\phantom{%
\begingroup \smaller\smaller\smaller\begin{tabular}{@{}c@{}}%
\phantom{0}\\\phantom{0}\\\phantom{0}
\end{tabular}\endgroup%
}}\!\right]$}%
{$\left[\!\llap{\phantom{%
\begingroup \smaller\smaller\smaller\begin{tabular}{@{}c@{}}%
0\\0\\0
\end{tabular}\endgroup%
}}\right.$}%
\begingroup \smaller\smaller\smaller\begin{tabular}{@{}c@{}}%
4\\-1\\-3
\end{tabular}\endgroup%
\kern3pt%
\begingroup \smaller\smaller\smaller\begin{tabular}{@{}c@{}}%
4\\-2\\-3
\end{tabular}\endgroup%
\kern3pt%
\begingroup \smaller\smaller\smaller\begin{tabular}{@{}c@{}}%
3\\-2\\-2
\end{tabular}\endgroup%
\kern3pt%
\begingroup \smaller\smaller\smaller\begin{tabular}{@{}c@{}}%
4\\-3\\-2
\end{tabular}\endgroup%
\kern3pt%
\begingroup \smaller\smaller\smaller\begin{tabular}{@{}c@{}}%
4\\-3\\-1
\end{tabular}\endgroup%
\kern3pt%
\begingroup \smaller\smaller\smaller\begin{tabular}{@{}c@{}}%
3\\-2\\0
\end{tabular}\endgroup%
\kern3pt%
\begingroup \smaller\smaller\smaller\begin{tabular}{@{}c@{}}%
4\\-2\\1
\end{tabular}\endgroup%
\kern3pt%
\begingroup \smaller\smaller\smaller\begin{tabular}{@{}c@{}}%
12\\-4\\5
\end{tabular}\endgroup%
\kern3pt%
\begingroup \smaller\smaller\smaller\begin{tabular}{@{}c@{}}%
36\\-7\\19
\end{tabular}\endgroup%
\kern3pt%
\begingroup \smaller\smaller\smaller\begin{tabular}{@{}c@{}}%
3\\0\\2
\end{tabular}\endgroup%
{$\left.\llap{\phantom{%
\begingroup \smaller\smaller\smaller\begin{tabular}{@{}c@{}}%
0\\0\\0
\end{tabular}\endgroup%
}}\!\right]$}%
}%
\ifdim\wd\matricesbox>\halfwidth\myboxwidth=\hsize\else\myboxwidth=\halfwidth\fi
\vbox{%
\ifdim\myboxwidth=\hsize
\setbox\onelinebox=\hbox{%
\vbox{\hbox{%
$\Pi_{20,23}$ spans $L_{251.3}$%
}\hbox{%
$32232226223223222622\rtimes C_{2}$%
}%
}%
\hfill\copy\matricesbox
}%
\ifdim\wd\onelinebox>\myboxwidth
\hbox to \myboxwidth{%
$\Pi_{20,23}$ spans $L_{251.3}$%
\hfil
$32232226223223222622\rtimes C_{2}$%
}%
\box\matricesbox
\else
\hbox to \myboxwidth{%
\unhbox\onelinebox
}%
\fi
\else
\hbox to \myboxwidth{%
$\Pi_{20,23}$ spans $L_{251.3}$%
\hfil}%
\hbox to \myboxwidth{%
$32232226223223222622\rtimes C_{2}$%
\hfil}%
\box\matricesbox
\fi
}%
\hfill\discretionary{}{}{}%
\setbox\matricesbox=\hbox{%
{$\left[\!\llap{\phantom{%
\begingroup \smaller\smaller\smaller
\endgroup%
}}\!\right]$}%
}%
\ifdim\wd\matricesbox>\halfwidth\myboxwidth=\hsize\else\myboxwidth=\halfwidth\fi
\vbox{%
\ifdim\myboxwidth=\hsize
\setbox\onelinebox=\hbox{%
\vbox{\hbox{%
$\Pi_{20,24}$ spans $L_{16.13}$%
}\hbox{%
$36322222363636363636$%
}%
}%
\hfill\copy\matricesbox
}%
\ifdim\wd\onelinebox>\myboxwidth
\hbox to \myboxwidth{%
$\Pi_{20,24}$ spans $L_{16.13}$%
\hfil
$36322222363636363636$%
}%
\box\matricesbox
\else
\hbox to \myboxwidth{%
\unhbox\onelinebox
}%
\fi
\else
\hbox to \myboxwidth{%
$\Pi_{20,24}$ spans $L_{16.13}$%
\hfil}%
\hbox to \myboxwidth{%
$36322222363636363636$%
\hfil}%
\box\matricesbox
\fi
}%
\hfill\discretionary{}{}{}%
\setbox\matricesbox=\hbox{%
{$\left[\!\llap{\phantom{%
\begingroup \smaller\smaller\smaller
\endgroup%
}}\!\right]$}%
}%
\ifdim\wd\matricesbox>\halfwidth\myboxwidth=\hsize\else\myboxwidth=\halfwidth\fi
\vbox{%
\ifdim\myboxwidth=\hsize
\setbox\onelinebox=\hbox{%
\vbox{\hbox{%
$\Pi_{20,25}$ spans $L_{16.13}$%
}\hbox{%
$36322223223636363636$%
}%
}%
\hfill\copy\matricesbox
}%
\ifdim\wd\onelinebox>\myboxwidth
\hbox to \myboxwidth{%
$\Pi_{20,25}$ spans $L_{16.13}$%
\hfil
$36322223223636363636$%
}%
\box\matricesbox
\else
\hbox to \myboxwidth{%
\unhbox\onelinebox
}%
\fi
\else
\hbox to \myboxwidth{%
$\Pi_{20,25}$ spans $L_{16.13}$%
\hfil}%
\hbox to \myboxwidth{%
$36322223223636363636$%
\hfil}%
\box\matricesbox
\fi
}%
\hfill\discretionary{}{}{}%
\setbox\matricesbox=\hbox{%
{$\left[\!\llap{\phantom{%
\begingroup \smaller\smaller\smaller
\endgroup%
}}\!\right]$}%
}%
\ifdim\wd\matricesbox>\halfwidth\myboxwidth=\hsize\else\myboxwidth=\halfwidth\fi
\vbox{%
\ifdim\myboxwidth=\hsize
\setbox\onelinebox=\hbox{%
\vbox{\hbox{%
$\Pi_{20,26}$ spans $L_{16.13}$%
}\hbox{%
$36322223632236363636$%
}%
}%
\hfill\copy\matricesbox
}%
\ifdim\wd\onelinebox>\myboxwidth
\hbox to \myboxwidth{%
$\Pi_{20,26}$ spans $L_{16.13}$%
\hfil
$36322223632236363636$%
}%
\box\matricesbox
\else
\hbox to \myboxwidth{%
\unhbox\onelinebox
}%
\fi
\else
\hbox to \myboxwidth{%
$\Pi_{20,26}$ spans $L_{16.13}$%
\hfil}%
\hbox to \myboxwidth{%
$36322223632236363636$%
\hfil}%
\box\matricesbox
\fi
}%
\hfill\discretionary{}{}{}%
\setbox\matricesbox=\hbox{%
{$\left[\!\llap{\phantom{%
\begingroup \smaller\smaller\smaller
\endgroup%
}}\!\right]$}%
}%
\ifdim\wd\matricesbox>\halfwidth\myboxwidth=\hsize\else\myboxwidth=\halfwidth\fi
\vbox{%
\ifdim\myboxwidth=\hsize
\setbox\onelinebox=\hbox{%
\vbox{\hbox{%
$\Pi_{20,27}$ spans $L_{16.13}$%
}\hbox{%
$36322223636322363636$%
}%
}%
\hfill\copy\matricesbox
}%
\ifdim\wd\onelinebox>\myboxwidth
\hbox to \myboxwidth{%
$\Pi_{20,27}$ spans $L_{16.13}$%
\hfil
$36322223636322363636$%
}%
\box\matricesbox
\else
\hbox to \myboxwidth{%
\unhbox\onelinebox
}%
\fi
\else
\hbox to \myboxwidth{%
$\Pi_{20,27}$ spans $L_{16.13}$%
\hfil}%
\hbox to \myboxwidth{%
$36322223636322363636$%
\hfil}%
\box\matricesbox
\fi
}%
\hfill\discretionary{}{}{}%
\setbox\matricesbox=\hbox{%
{$\left[\!\llap{\phantom{%
\begingroup \smaller\smaller\smaller
\endgroup%
}}\!\right]$}%
}%
\ifdim\wd\matricesbox>\halfwidth\myboxwidth=\hsize\else\myboxwidth=\halfwidth\fi
\vbox{%
\ifdim\myboxwidth=\hsize
\setbox\onelinebox=\hbox{%
\vbox{\hbox{%
$\Pi_{20,28}$ spans $L_{16.13}$%
}\hbox{%
$36322232232236363636$%
}%
}%
\hfill\copy\matricesbox
}%
\ifdim\wd\onelinebox>\myboxwidth
\hbox to \myboxwidth{%
$\Pi_{20,28}$ spans $L_{16.13}$%
\hfil
$36322232232236363636$%
}%
\box\matricesbox
\else
\hbox to \myboxwidth{%
\unhbox\onelinebox
}%
\fi
\else
\hbox to \myboxwidth{%
$\Pi_{20,28}$ spans $L_{16.13}$%
\hfil}%
\hbox to \myboxwidth{%
$36322232232236363636$%
\hfil}%
\box\matricesbox
\fi
}%
\hfill\discretionary{}{}{}%
\setbox\matricesbox=\hbox{%
{$\left[\!\llap{\phantom{%
\begingroup \smaller\smaller\smaller
\endgroup%
}}\!\right]$}%
}%
\ifdim\wd\matricesbox>\halfwidth\myboxwidth=\hsize\else\myboxwidth=\halfwidth\fi
\vbox{%
\ifdim\myboxwidth=\hsize
\setbox\onelinebox=\hbox{%
\vbox{\hbox{%
$\Pi_{20,29}$ spans $L_{16.13}$%
}\hbox{%
$36322232236322363636$%
}%
}%
\hfill\copy\matricesbox
}%
\ifdim\wd\onelinebox>\myboxwidth
\hbox to \myboxwidth{%
$\Pi_{20,29}$ spans $L_{16.13}$%
\hfil
$36322232236322363636$%
}%
\box\matricesbox
\else
\hbox to \myboxwidth{%
\unhbox\onelinebox
}%
\fi
\else
\hbox to \myboxwidth{%
$\Pi_{20,29}$ spans $L_{16.13}$%
\hfil}%
\hbox to \myboxwidth{%
$36322232236322363636$%
\hfil}%
\box\matricesbox
\fi
}%
\hfill\discretionary{}{}{}%
\setbox\matricesbox=\hbox{%
{$\left[\!\llap{\phantom{%
\begingroup \smaller\smaller\smaller
\endgroup%
}}\!\right]$}%
}%
\ifdim\wd\matricesbox>\halfwidth\myboxwidth=\hsize\else\myboxwidth=\halfwidth\fi
\vbox{%
\ifdim\myboxwidth=\hsize
\setbox\onelinebox=\hbox{%
\vbox{\hbox{%
$\Pi_{20,30}$ spans $L_{16.13}$%
}\hbox{%
$36322232236363223636$%
}%
}%
\hfill\copy\matricesbox
}%
\ifdim\wd\onelinebox>\myboxwidth
\hbox to \myboxwidth{%
$\Pi_{20,30}$ spans $L_{16.13}$%
\hfil
$36322232236363223636$%
}%
\box\matricesbox
\else
\hbox to \myboxwidth{%
\unhbox\onelinebox
}%
\fi
\else
\hbox to \myboxwidth{%
$\Pi_{20,30}$ spans $L_{16.13}$%
\hfil}%
\hbox to \myboxwidth{%
$36322232236363223636$%
\hfil}%
\box\matricesbox
\fi
}%
\hfill\discretionary{}{}{}%
\setbox\matricesbox=\hbox{%
{$\left[\!\llap{\phantom{%
\begingroup \smaller\smaller\smaller
\endgroup%
}}\!\right]$}%
}%
\ifdim\wd\matricesbox>\halfwidth\myboxwidth=\hsize\else\myboxwidth=\halfwidth\fi
\vbox{%
\ifdim\myboxwidth=\hsize
\setbox\onelinebox=\hbox{%
\vbox{\hbox{%
$\Pi_{20,31}$ spans $L_{16.13}$%
}\hbox{%
$36322232236363632236$%
}%
}%
\hfill\copy\matricesbox
}%
\ifdim\wd\onelinebox>\myboxwidth
\hbox to \myboxwidth{%
$\Pi_{20,31}$ spans $L_{16.13}$%
\hfil
$36322232236363632236$%
}%
\box\matricesbox
\else
\hbox to \myboxwidth{%
\unhbox\onelinebox
}%
\fi
\else
\hbox to \myboxwidth{%
$\Pi_{20,31}$ spans $L_{16.13}$%
\hfil}%
\hbox to \myboxwidth{%
$36322232236363632236$%
\hfil}%
\box\matricesbox
\fi
}%
\hfill\discretionary{}{}{}%
\setbox\matricesbox=\hbox{%
{$\left[\!\llap{\phantom{%
\begingroup \smaller\smaller\smaller
\endgroup%
}}\!\right]$}%
}%
\ifdim\wd\matricesbox>\halfwidth\myboxwidth=\hsize\else\myboxwidth=\halfwidth\fi
\vbox{%
\ifdim\myboxwidth=\hsize
\setbox\onelinebox=\hbox{%
\vbox{\hbox{%
$\Pi_{20,32}$ spans $L_{16.13}$%
}\hbox{%
$36322232236363636322$%
}%
}%
\hfill\copy\matricesbox
}%
\ifdim\wd\onelinebox>\myboxwidth
\hbox to \myboxwidth{%
$\Pi_{20,32}$ spans $L_{16.13}$%
\hfil
$36322232236363636322$%
}%
\box\matricesbox
\else
\hbox to \myboxwidth{%
\unhbox\onelinebox
}%
\fi
\else
\hbox to \myboxwidth{%
$\Pi_{20,32}$ spans $L_{16.13}$%
\hfil}%
\hbox to \myboxwidth{%
$36322232236363636322$%
\hfil}%
\box\matricesbox
\fi
}%
\hfill\discretionary{}{}{}%
\setbox\matricesbox=\hbox{%
{$\left[\!\llap{\phantom{%
\begingroup \smaller\smaller\smaller
\endgroup%
}}\!\right]$}%
}%
\ifdim\wd\matricesbox>\halfwidth\myboxwidth=\hsize\else\myboxwidth=\halfwidth\fi
\vbox{%
\ifdim\myboxwidth=\hsize
\setbox\onelinebox=\hbox{%
\vbox{\hbox{%
$\Pi_{20,33}$ spans $L_{16.13}$%
}\hbox{%
$36322236322236363636$%
}%
}%
\hfill\copy\matricesbox
}%
\ifdim\wd\onelinebox>\myboxwidth
\hbox to \myboxwidth{%
$\Pi_{20,33}$ spans $L_{16.13}$%
\hfil
$36322236322236363636$%
}%
\box\matricesbox
\else
\hbox to \myboxwidth{%
\unhbox\onelinebox
}%
\fi
\else
\hbox to \myboxwidth{%
$\Pi_{20,33}$ spans $L_{16.13}$%
\hfil}%
\hbox to \myboxwidth{%
$36322236322236363636$%
\hfil}%
\box\matricesbox
\fi
}%
\hfill\discretionary{}{}{}%
\setbox\matricesbox=\hbox{%
{$\left[\!\llap{\phantom{%
\begingroup \smaller\smaller\smaller
\endgroup%
}}\!\right]$}%
}%
\ifdim\wd\matricesbox>\halfwidth\myboxwidth=\hsize\else\myboxwidth=\halfwidth\fi
\vbox{%
\ifdim\myboxwidth=\hsize
\setbox\onelinebox=\hbox{%
\vbox{\hbox{%
$\Pi_{20,34}$ spans $L_{16.13}$%
}\hbox{%
$36322236322322363636$%
}%
}%
\hfill\copy\matricesbox
}%
\ifdim\wd\onelinebox>\myboxwidth
\hbox to \myboxwidth{%
$\Pi_{20,34}$ spans $L_{16.13}$%
\hfil
$36322236322322363636$%
}%
\box\matricesbox
\else
\hbox to \myboxwidth{%
\unhbox\onelinebox
}%
\fi
\else
\hbox to \myboxwidth{%
$\Pi_{20,34}$ spans $L_{16.13}$%
\hfil}%
\hbox to \myboxwidth{%
$36322236322322363636$%
\hfil}%
\box\matricesbox
\fi
}%
\hfill\discretionary{}{}{}%
\setbox\matricesbox=\hbox{%
{$\left[\!\llap{\phantom{%
\begingroup \smaller\smaller\smaller
\endgroup%
}}\!\right]$}%
}%
\ifdim\wd\matricesbox>\halfwidth\myboxwidth=\hsize\else\myboxwidth=\halfwidth\fi
\vbox{%
\ifdim\myboxwidth=\hsize
\setbox\onelinebox=\hbox{%
\vbox{\hbox{%
$\Pi_{20,35}$ spans $L_{16.13}$%
}\hbox{%
$36322236322363223636$%
}%
}%
\hfill\copy\matricesbox
}%
\ifdim\wd\onelinebox>\myboxwidth
\hbox to \myboxwidth{%
$\Pi_{20,35}$ spans $L_{16.13}$%
\hfil
$36322236322363223636$%
}%
\box\matricesbox
\else
\hbox to \myboxwidth{%
\unhbox\onelinebox
}%
\fi
\else
\hbox to \myboxwidth{%
$\Pi_{20,35}$ spans $L_{16.13}$%
\hfil}%
\hbox to \myboxwidth{%
$36322236322363223636$%
\hfil}%
\box\matricesbox
\fi
}%
\hfill\discretionary{}{}{}%
\setbox\matricesbox=\hbox{%
{$\left[\!\llap{\phantom{%
\begingroup \smaller\smaller\smaller
\endgroup%
}}\!\right]$}%
}%
\ifdim\wd\matricesbox>\halfwidth\myboxwidth=\hsize\else\myboxwidth=\halfwidth\fi
\vbox{%
\ifdim\myboxwidth=\hsize
\setbox\onelinebox=\hbox{%
\vbox{\hbox{%
$\Pi_{20,36}$ spans $L_{16.13}$%
}\hbox{%
$36322236322363632236$%
}%
}%
\hfill\copy\matricesbox
}%
\ifdim\wd\onelinebox>\myboxwidth
\hbox to \myboxwidth{%
$\Pi_{20,36}$ spans $L_{16.13}$%
\hfil
$36322236322363632236$%
}%
\box\matricesbox
\else
\hbox to \myboxwidth{%
\unhbox\onelinebox
}%
\fi
\else
\hbox to \myboxwidth{%
$\Pi_{20,36}$ spans $L_{16.13}$%
\hfil}%
\hbox to \myboxwidth{%
$36322236322363632236$%
\hfil}%
\box\matricesbox
\fi
}%
\hfill\discretionary{}{}{}%
\setbox\matricesbox=\hbox{%
{$\left[\!\llap{\phantom{%
\begingroup \smaller\smaller\smaller
\endgroup%
}}\!\right]$}%
}%
\ifdim\wd\matricesbox>\halfwidth\myboxwidth=\hsize\else\myboxwidth=\halfwidth\fi
\vbox{%
\ifdim\myboxwidth=\hsize
\setbox\onelinebox=\hbox{%
\vbox{\hbox{%
$\Pi_{20,37}$ spans $L_{16.13}$%
}\hbox{%
$36322236322363636322$%
}%
}%
\hfill\copy\matricesbox
}%
\ifdim\wd\onelinebox>\myboxwidth
\hbox to \myboxwidth{%
$\Pi_{20,37}$ spans $L_{16.13}$%
\hfil
$36322236322363636322$%
}%
\box\matricesbox
\else
\hbox to \myboxwidth{%
\unhbox\onelinebox
}%
\fi
\else
\hbox to \myboxwidth{%
$\Pi_{20,37}$ spans $L_{16.13}$%
\hfil}%
\hbox to \myboxwidth{%
$36322236322363636322$%
\hfil}%
\box\matricesbox
\fi
}%
\hfill\discretionary{}{}{}%
\setbox\matricesbox=\hbox{%
{$\left[\!\llap{\phantom{%
\begingroup \smaller\smaller\smaller
\endgroup%
}}\!\right]$}%
}%
\ifdim\wd\matricesbox>\halfwidth\myboxwidth=\hsize\else\myboxwidth=\halfwidth\fi
\vbox{%
\ifdim\myboxwidth=\hsize
\setbox\onelinebox=\hbox{%
\vbox{\hbox{%
$\Pi_{20,38}$ spans $L_{16.13}$%
}\hbox{%
$36322236363223223636$%
}%
}%
\hfill\copy\matricesbox
}%
\ifdim\wd\onelinebox>\myboxwidth
\hbox to \myboxwidth{%
$\Pi_{20,38}$ spans $L_{16.13}$%
\hfil
$36322236363223223636$%
}%
\box\matricesbox
\else
\hbox to \myboxwidth{%
\unhbox\onelinebox
}%
\fi
\else
\hbox to \myboxwidth{%
$\Pi_{20,38}$ spans $L_{16.13}$%
\hfil}%
\hbox to \myboxwidth{%
$36322236363223223636$%
\hfil}%
\box\matricesbox
\fi
}%
\hfill\discretionary{}{}{}%
\setbox\matricesbox=\hbox{%
{$\left[\!\llap{\phantom{%
\begingroup \smaller\smaller\smaller
\endgroup%
}}\!\right]$}%
}%
\ifdim\wd\matricesbox>\halfwidth\myboxwidth=\hsize\else\myboxwidth=\halfwidth\fi
\vbox{%
\ifdim\myboxwidth=\hsize
\setbox\onelinebox=\hbox{%
\vbox{\hbox{%
$\Pi_{20,39}$ spans $L_{16.13}$%
}\hbox{%
$36322236363223632236$%
}%
}%
\hfill\copy\matricesbox
}%
\ifdim\wd\onelinebox>\myboxwidth
\hbox to \myboxwidth{%
$\Pi_{20,39}$ spans $L_{16.13}$%
\hfil
$36322236363223632236$%
}%
\box\matricesbox
\else
\hbox to \myboxwidth{%
\unhbox\onelinebox
}%
\fi
\else
\hbox to \myboxwidth{%
$\Pi_{20,39}$ spans $L_{16.13}$%
\hfil}%
\hbox to \myboxwidth{%
$36322236363223632236$%
\hfil}%
\box\matricesbox
\fi
}%
\hfill\discretionary{}{}{}%
\setbox\matricesbox=\hbox{%
{$\left[\!\llap{\phantom{%
\begingroup \smaller\smaller\smaller
\endgroup%
}}\!\right]$}%
}%
\ifdim\wd\matricesbox>\halfwidth\myboxwidth=\hsize\else\myboxwidth=\halfwidth\fi
\vbox{%
\ifdim\myboxwidth=\hsize
\setbox\onelinebox=\hbox{%
\vbox{\hbox{%
$\Pi_{20,40}$ spans $L_{16.13}$%
}\hbox{%
$36322236363223636322$%
}%
}%
\hfill\copy\matricesbox
}%
\ifdim\wd\onelinebox>\myboxwidth
\hbox to \myboxwidth{%
$\Pi_{20,40}$ spans $L_{16.13}$%
\hfil
$36322236363223636322$%
}%
\box\matricesbox
\else
\hbox to \myboxwidth{%
\unhbox\onelinebox
}%
\fi
\else
\hbox to \myboxwidth{%
$\Pi_{20,40}$ spans $L_{16.13}$%
\hfil}%
\hbox to \myboxwidth{%
$36322236363223636322$%
\hfil}%
\box\matricesbox
\fi
}%
\hfill\discretionary{}{}{}%
\setbox\matricesbox=\hbox{%
{$\left[\!\llap{\phantom{%
\begingroup \smaller\smaller\smaller
\endgroup%
}}\!\right]$}%
}%
\ifdim\wd\matricesbox>\halfwidth\myboxwidth=\hsize\else\myboxwidth=\halfwidth\fi
\vbox{%
\ifdim\myboxwidth=\hsize
\setbox\onelinebox=\hbox{%
\vbox{\hbox{%
$\Pi_{20,41}$ spans $L_{16.13}$%
}\hbox{%
$36322236363632232236$%
}%
}%
\hfill\copy\matricesbox
}%
\ifdim\wd\onelinebox>\myboxwidth
\hbox to \myboxwidth{%
$\Pi_{20,41}$ spans $L_{16.13}$%
\hfil
$36322236363632232236$%
}%
\box\matricesbox
\else
\hbox to \myboxwidth{%
\unhbox\onelinebox
}%
\fi
\else
\hbox to \myboxwidth{%
$\Pi_{20,41}$ spans $L_{16.13}$%
\hfil}%
\hbox to \myboxwidth{%
$36322236363632232236$%
\hfil}%
\box\matricesbox
\fi
}%
\hfill\discretionary{}{}{}%
\setbox\matricesbox=\hbox{%
{$\left[\!\llap{\phantom{%
\begingroup \smaller\smaller\smaller
\endgroup%
}}\!\right]$}%
}%
\ifdim\wd\matricesbox>\halfwidth\myboxwidth=\hsize\else\myboxwidth=\halfwidth\fi
\vbox{%
\ifdim\myboxwidth=\hsize
\setbox\onelinebox=\hbox{%
\vbox{\hbox{%
$\Pi_{20,42}$ spans $L_{16.13}$%
}\hbox{%
$36322236363632236322$%
}%
}%
\hfill\copy\matricesbox
}%
\ifdim\wd\onelinebox>\myboxwidth
\hbox to \myboxwidth{%
$\Pi_{20,42}$ spans $L_{16.13}$%
\hfil
$36322236363632236322$%
}%
\box\matricesbox
\else
\hbox to \myboxwidth{%
\unhbox\onelinebox
}%
\fi
\else
\hbox to \myboxwidth{%
$\Pi_{20,42}$ spans $L_{16.13}$%
\hfil}%
\hbox to \myboxwidth{%
$36322236363632236322$%
\hfil}%
\box\matricesbox
\fi
}%
\hfill\discretionary{}{}{}%
\setbox\matricesbox=\hbox{%
{$\left[\!\llap{\phantom{%
\begingroup \smaller\smaller\smaller
\endgroup%
}}\!\right]$}%
}%
\ifdim\wd\matricesbox>\halfwidth\myboxwidth=\hsize\else\myboxwidth=\halfwidth\fi
\vbox{%
\ifdim\myboxwidth=\hsize
\setbox\onelinebox=\hbox{%
\vbox{\hbox{%
$\Pi_{20,43}$ spans $L_{16.13}$%
}\hbox{%
$36322236363636322322$%
}%
}%
\hfill\copy\matricesbox
}%
\ifdim\wd\onelinebox>\myboxwidth
\hbox to \myboxwidth{%
$\Pi_{20,43}$ spans $L_{16.13}$%
\hfil
$36322236363636322322$%
}%
\box\matricesbox
\else
\hbox to \myboxwidth{%
\unhbox\onelinebox
}%
\fi
\else
\hbox to \myboxwidth{%
$\Pi_{20,43}$ spans $L_{16.13}$%
\hfil}%
\hbox to \myboxwidth{%
$36322236363636322322$%
\hfil}%
\box\matricesbox
\fi
}%
\hfill\discretionary{}{}{}%
\setbox\matricesbox=\hbox{%
{$\left[\!\llap{\phantom{%
\begingroup \smaller\smaller\smaller
\endgroup%
}}\!\right]$}%
}%
\ifdim\wd\matricesbox>\halfwidth\myboxwidth=\hsize\else\myboxwidth=\halfwidth\fi
\vbox{%
\ifdim\myboxwidth=\hsize
\setbox\onelinebox=\hbox{%
\vbox{\hbox{%
$\Pi_{20,44}$ spans $L_{16.13}$%
}\hbox{%
$36322322232236363636$%
}%
}%
\hfill\copy\matricesbox
}%
\ifdim\wd\onelinebox>\myboxwidth
\hbox to \myboxwidth{%
$\Pi_{20,44}$ spans $L_{16.13}$%
\hfil
$36322322232236363636$%
}%
\box\matricesbox
\else
\hbox to \myboxwidth{%
\unhbox\onelinebox
}%
\fi
\else
\hbox to \myboxwidth{%
$\Pi_{20,44}$ spans $L_{16.13}$%
\hfil}%
\hbox to \myboxwidth{%
$36322322232236363636$%
\hfil}%
\box\matricesbox
\fi
}%
\hfill\discretionary{}{}{}%
\setbox\matricesbox=\hbox{%
{$\left[\!\llap{\phantom{%
\begingroup \smaller\smaller\smaller
\endgroup%
}}\!\right]$}%
}%
\ifdim\wd\matricesbox>\halfwidth\myboxwidth=\hsize\else\myboxwidth=\halfwidth\fi
\vbox{%
\ifdim\myboxwidth=\hsize
\setbox\onelinebox=\hbox{%
\vbox{\hbox{%
$\Pi_{20,45}$ spans $L_{16.13}$%
}\hbox{%
$36322322236322363636$%
}%
}%
\hfill\copy\matricesbox
}%
\ifdim\wd\onelinebox>\myboxwidth
\hbox to \myboxwidth{%
$\Pi_{20,45}$ spans $L_{16.13}$%
\hfil
$36322322236322363636$%
}%
\box\matricesbox
\else
\hbox to \myboxwidth{%
\unhbox\onelinebox
}%
\fi
\else
\hbox to \myboxwidth{%
$\Pi_{20,45}$ spans $L_{16.13}$%
\hfil}%
\hbox to \myboxwidth{%
$36322322236322363636$%
\hfil}%
\box\matricesbox
\fi
}%
\hfill\discretionary{}{}{}%
\setbox\matricesbox=\hbox{%
{$\left[\!\llap{\phantom{%
\begingroup \smaller\smaller\smaller
\endgroup%
}}\!\right]$}%
}%
\ifdim\wd\matricesbox>\halfwidth\myboxwidth=\hsize\else\myboxwidth=\halfwidth\fi
\vbox{%
\ifdim\myboxwidth=\hsize
\setbox\onelinebox=\hbox{%
\vbox{\hbox{%
$\Pi_{20,46}$ spans $L_{16.13}$%
}\hbox{%
$36322322236363223636$%
}%
}%
\hfill\copy\matricesbox
}%
\ifdim\wd\onelinebox>\myboxwidth
\hbox to \myboxwidth{%
$\Pi_{20,46}$ spans $L_{16.13}$%
\hfil
$36322322236363223636$%
}%
\box\matricesbox
\else
\hbox to \myboxwidth{%
\unhbox\onelinebox
}%
\fi
\else
\hbox to \myboxwidth{%
$\Pi_{20,46}$ spans $L_{16.13}$%
\hfil}%
\hbox to \myboxwidth{%
$36322322236363223636$%
\hfil}%
\box\matricesbox
\fi
}%
\hfill\discretionary{}{}{}%
\setbox\matricesbox=\hbox{%
{$\left[\!\llap{\phantom{%
\begingroup \smaller\smaller\smaller
\endgroup%
}}\!\right]$}%
}%
\ifdim\wd\matricesbox>\halfwidth\myboxwidth=\hsize\else\myboxwidth=\halfwidth\fi
\vbox{%
\ifdim\myboxwidth=\hsize
\setbox\onelinebox=\hbox{%
\vbox{\hbox{%
$\Pi_{20,47}$ spans $L_{16.13}$%
}\hbox{%
$36322322236363632236$%
}%
}%
\hfill\copy\matricesbox
}%
\ifdim\wd\onelinebox>\myboxwidth
\hbox to \myboxwidth{%
$\Pi_{20,47}$ spans $L_{16.13}$%
\hfil
$36322322236363632236$%
}%
\box\matricesbox
\else
\hbox to \myboxwidth{%
\unhbox\onelinebox
}%
\fi
\else
\hbox to \myboxwidth{%
$\Pi_{20,47}$ spans $L_{16.13}$%
\hfil}%
\hbox to \myboxwidth{%
$36322322236363632236$%
\hfil}%
\box\matricesbox
\fi
}%
\hfill\discretionary{}{}{}%
\setbox\matricesbox=\hbox{%
{$\left[\!\llap{\phantom{%
\begingroup \smaller\smaller\smaller
\endgroup%
}}\!\right]$}%
}%
\ifdim\wd\matricesbox>\halfwidth\myboxwidth=\hsize\else\myboxwidth=\halfwidth\fi
\vbox{%
\ifdim\myboxwidth=\hsize
\setbox\onelinebox=\hbox{%
\vbox{\hbox{%
$\Pi_{20,48}$ spans $L_{16.13}$%
}\hbox{%
$36322322236363636322$%
}%
}%
\hfill\copy\matricesbox
}%
\ifdim\wd\onelinebox>\myboxwidth
\hbox to \myboxwidth{%
$\Pi_{20,48}$ spans $L_{16.13}$%
\hfil
$36322322236363636322$%
}%
\box\matricesbox
\else
\hbox to \myboxwidth{%
\unhbox\onelinebox
}%
\fi
\else
\hbox to \myboxwidth{%
$\Pi_{20,48}$ spans $L_{16.13}$%
\hfil}%
\hbox to \myboxwidth{%
$36322322236363636322$%
\hfil}%
\box\matricesbox
\fi
}%
\hfill\discretionary{}{}{}%
\setbox\matricesbox=\hbox{%
{$\left[\!\llap{\phantom{%
\begingroup \smaller\smaller\smaller
\endgroup%
}}\!\right]$}%
}%
\ifdim\wd\matricesbox>\halfwidth\myboxwidth=\hsize\else\myboxwidth=\halfwidth\fi
\vbox{%
\ifdim\myboxwidth=\hsize
\setbox\onelinebox=\hbox{%
\vbox{\hbox{%
$\Pi_{20,49}$ spans $L_{16.13}$%
}\hbox{%
$36322322322236363636$%
}%
}%
\hfill\copy\matricesbox
}%
\ifdim\wd\onelinebox>\myboxwidth
\hbox to \myboxwidth{%
$\Pi_{20,49}$ spans $L_{16.13}$%
\hfil
$36322322322236363636$%
}%
\box\matricesbox
\else
\hbox to \myboxwidth{%
\unhbox\onelinebox
}%
\fi
\else
\hbox to \myboxwidth{%
$\Pi_{20,49}$ spans $L_{16.13}$%
\hfil}%
\hbox to \myboxwidth{%
$36322322322236363636$%
\hfil}%
\box\matricesbox
\fi
}%
\hfill\discretionary{}{}{}%
\setbox\matricesbox=\hbox{%
{$\left[\!\llap{\phantom{%
\begingroup \smaller\smaller\smaller
\endgroup%
}}\!\right]$}%
}%
\ifdim\wd\matricesbox>\halfwidth\myboxwidth=\hsize\else\myboxwidth=\halfwidth\fi
\vbox{%
\ifdim\myboxwidth=\hsize
\setbox\onelinebox=\hbox{%
\vbox{\hbox{%
$\Pi_{20,50}$ spans $L_{16.13}$%
}\hbox{%
$36322322322363223636$%
}%
}%
\hfill\copy\matricesbox
}%
\ifdim\wd\onelinebox>\myboxwidth
\hbox to \myboxwidth{%
$\Pi_{20,50}$ spans $L_{16.13}$%
\hfil
$36322322322363223636$%
}%
\box\matricesbox
\else
\hbox to \myboxwidth{%
\unhbox\onelinebox
}%
\fi
\else
\hbox to \myboxwidth{%
$\Pi_{20,50}$ spans $L_{16.13}$%
\hfil}%
\hbox to \myboxwidth{%
$36322322322363223636$%
\hfil}%
\box\matricesbox
\fi
}%
\hfill\discretionary{}{}{}%
\setbox\matricesbox=\hbox{%
{$\left[\!\llap{\phantom{%
\begingroup \smaller\smaller\smaller
\endgroup%
}}\!\right]$}%
}%
\ifdim\wd\matricesbox>\halfwidth\myboxwidth=\hsize\else\myboxwidth=\halfwidth\fi
\vbox{%
\ifdim\myboxwidth=\hsize
\setbox\onelinebox=\hbox{%
\vbox{\hbox{%
$\Pi_{20,51}$ spans $L_{16.13}$%
}\hbox{%
$36322322363223223636$%
}%
}%
\hfill\copy\matricesbox
}%
\ifdim\wd\onelinebox>\myboxwidth
\hbox to \myboxwidth{%
$\Pi_{20,51}$ spans $L_{16.13}$%
\hfil
$36322322363223223636$%
}%
\box\matricesbox
\else
\hbox to \myboxwidth{%
\unhbox\onelinebox
}%
\fi
\else
\hbox to \myboxwidth{%
$\Pi_{20,51}$ spans $L_{16.13}$%
\hfil}%
\hbox to \myboxwidth{%
$36322322363223223636$%
\hfil}%
\box\matricesbox
\fi
}%
\hfill\discretionary{}{}{}%
\setbox\matricesbox=\hbox{%
{$\left[\!\llap{\phantom{%
\begingroup \smaller\smaller\smaller
\endgroup%
}}\!\right]$}%
}%
\ifdim\wd\matricesbox>\halfwidth\myboxwidth=\hsize\else\myboxwidth=\halfwidth\fi
\vbox{%
\ifdim\myboxwidth=\hsize
\setbox\onelinebox=\hbox{%
\vbox{\hbox{%
$\Pi_{20,52}$ spans $L_{16.13}$%
}\hbox{%
$36322322363223632236$%
}%
}%
\hfill\copy\matricesbox
}%
\ifdim\wd\onelinebox>\myboxwidth
\hbox to \myboxwidth{%
$\Pi_{20,52}$ spans $L_{16.13}$%
\hfil
$36322322363223632236$%
}%
\box\matricesbox
\else
\hbox to \myboxwidth{%
\unhbox\onelinebox
}%
\fi
\else
\hbox to \myboxwidth{%
$\Pi_{20,52}$ spans $L_{16.13}$%
\hfil}%
\hbox to \myboxwidth{%
$36322322363223632236$%
\hfil}%
\box\matricesbox
\fi
}%
\hfill\discretionary{}{}{}%
\setbox\matricesbox=\hbox{%
{$\left[\!\llap{\phantom{%
\begingroup \smaller\smaller\smaller
\endgroup%
}}\!\right]$}%
}%
\ifdim\wd\matricesbox>\halfwidth\myboxwidth=\hsize\else\myboxwidth=\halfwidth\fi
\vbox{%
\ifdim\myboxwidth=\hsize
\setbox\onelinebox=\hbox{%
\vbox{\hbox{%
$\Pi_{20,53}$ spans $L_{16.13}$%
}\hbox{%
$36322363222236363636$%
}%
}%
\hfill\copy\matricesbox
}%
\ifdim\wd\onelinebox>\myboxwidth
\hbox to \myboxwidth{%
$\Pi_{20,53}$ spans $L_{16.13}$%
\hfil
$36322363222236363636$%
}%
\box\matricesbox
\else
\hbox to \myboxwidth{%
\unhbox\onelinebox
}%
\fi
\else
\hbox to \myboxwidth{%
$\Pi_{20,53}$ spans $L_{16.13}$%
\hfil}%
\hbox to \myboxwidth{%
$36322363222236363636$%
\hfil}%
\box\matricesbox
\fi
}%
\hfill\discretionary{}{}{}%
\setbox\matricesbox=\hbox{%
{$\left[\!\llap{\phantom{%
\begingroup \smaller\smaller\smaller
\endgroup%
}}\!\right]$}%
}%
\ifdim\wd\matricesbox>\halfwidth\myboxwidth=\hsize\else\myboxwidth=\halfwidth\fi
\vbox{%
\ifdim\myboxwidth=\hsize
\setbox\onelinebox=\hbox{%
\vbox{\hbox{%
$\Pi_{20,54}$ spans $L_{16.13}$%
}\hbox{%
$36322363223223223636$%
}%
}%
\hfill\copy\matricesbox
}%
\ifdim\wd\onelinebox>\myboxwidth
\hbox to \myboxwidth{%
$\Pi_{20,54}$ spans $L_{16.13}$%
\hfil
$36322363223223223636$%
}%
\box\matricesbox
\else
\hbox to \myboxwidth{%
\unhbox\onelinebox
}%
\fi
\else
\hbox to \myboxwidth{%
$\Pi_{20,54}$ spans $L_{16.13}$%
\hfil}%
\hbox to \myboxwidth{%
$36322363223223223636$%
\hfil}%
\box\matricesbox
\fi
}%
\hfill\discretionary{}{}{}%
\setbox\matricesbox=\hbox{%
{$\left[\!\llap{\phantom{%
\begingroup \smaller\smaller\smaller
\endgroup%
}}\!\right]$}%
}%
\ifdim\wd\matricesbox>\halfwidth\myboxwidth=\hsize\else\myboxwidth=\halfwidth\fi
\vbox{%
\ifdim\myboxwidth=\hsize
\setbox\onelinebox=\hbox{%
\vbox{\hbox{%
$\Pi_{20,55}$ spans $L_{16.13}$%
}\hbox{%
$36322363223223636322$%
}%
}%
\hfill\copy\matricesbox
}%
\ifdim\wd\onelinebox>\myboxwidth
\hbox to \myboxwidth{%
$\Pi_{20,55}$ spans $L_{16.13}$%
\hfil
$36322363223223636322$%
}%
\box\matricesbox
\else
\hbox to \myboxwidth{%
\unhbox\onelinebox
}%
\fi
\else
\hbox to \myboxwidth{%
$\Pi_{20,55}$ spans $L_{16.13}$%
\hfil}%
\hbox to \myboxwidth{%
$36322363223223636322$%
\hfil}%
\box\matricesbox
\fi
}%
\hfill\discretionary{}{}{}%
\setbox\matricesbox=\hbox{%
{$\left[\!\llap{\phantom{%
\begingroup \smaller\smaller\smaller
\endgroup%
}}\!\right]$}%
}%
\ifdim\wd\matricesbox>\halfwidth\myboxwidth=\hsize\else\myboxwidth=\halfwidth\fi
\vbox{%
\ifdim\myboxwidth=\hsize
\setbox\onelinebox=\hbox{%
\vbox{\hbox{%
$\Pi_{20,56}$ spans $L_{16.13}$%
}\hbox{%
$36363222232236363636$%
}%
}%
\hfill\copy\matricesbox
}%
\ifdim\wd\onelinebox>\myboxwidth
\hbox to \myboxwidth{%
$\Pi_{20,56}$ spans $L_{16.13}$%
\hfil
$36363222232236363636$%
}%
\box\matricesbox
\else
\hbox to \myboxwidth{%
\unhbox\onelinebox
}%
\fi
\else
\hbox to \myboxwidth{%
$\Pi_{20,56}$ spans $L_{16.13}$%
\hfil}%
\hbox to \myboxwidth{%
$36363222232236363636$%
\hfil}%
\box\matricesbox
\fi
}%
\hfill\discretionary{}{}{}%
\setbox\matricesbox=\hbox{%
{$\left[\!\llap{\phantom{%
\begingroup \smaller\smaller\smaller
\endgroup%
}}\!\right]$}%
}%
\ifdim\wd\matricesbox>\halfwidth\myboxwidth=\hsize\else\myboxwidth=\halfwidth\fi
\vbox{%
\ifdim\myboxwidth=\hsize
\setbox\onelinebox=\hbox{%
\vbox{\hbox{%
$\Pi_{20,57}$ spans $L_{16.13}$%
}\hbox{%
$36363222236363223636$%
}%
}%
\hfill\copy\matricesbox
}%
\ifdim\wd\onelinebox>\myboxwidth
\hbox to \myboxwidth{%
$\Pi_{20,57}$ spans $L_{16.13}$%
\hfil
$36363222236363223636$%
}%
\box\matricesbox
\else
\hbox to \myboxwidth{%
\unhbox\onelinebox
}%
\fi
\else
\hbox to \myboxwidth{%
$\Pi_{20,57}$ spans $L_{16.13}$%
\hfil}%
\hbox to \myboxwidth{%
$36363222236363223636$%
\hfil}%
\box\matricesbox
\fi
}%
\hfill\discretionary{}{}{}%
\setbox\matricesbox=\hbox{%
{$\left[\!\llap{\phantom{%
\begingroup \smaller\smaller\smaller
\endgroup%
}}\!\right]$}%
}%
\ifdim\wd\matricesbox>\halfwidth\myboxwidth=\hsize\else\myboxwidth=\halfwidth\fi
\vbox{%
\ifdim\myboxwidth=\hsize
\setbox\onelinebox=\hbox{%
\vbox{\hbox{%
$\Pi_{20,58}$ spans $L_{251.3}$%
}\hbox{%
$32232232232226222622$%
}%
}%
\hfill\copy\matricesbox
}%
\ifdim\wd\onelinebox>\myboxwidth
\hbox to \myboxwidth{%
$\Pi_{20,58}$ spans $L_{251.3}$%
\hfil
$32232232232226222622$%
}%
\box\matricesbox
\else
\hbox to \myboxwidth{%
\unhbox\onelinebox
}%
\fi
\else
\hbox to \myboxwidth{%
$\Pi_{20,58}$ spans $L_{251.3}$%
\hfil}%
\hbox to \myboxwidth{%
$32232232232226222622$%
\hfil}%
\box\matricesbox
\fi
}%
\hfill\discretionary{}{}{}%
\setbox\matricesbox=\hbox{%
{$\left[\!\llap{\phantom{%
\begingroup \smaller\smaller\smaller
\endgroup%
}}\!\right]$}%
}%
\ifdim\wd\matricesbox>\halfwidth\myboxwidth=\hsize\else\myboxwidth=\halfwidth\fi
\vbox{%
\ifdim\myboxwidth=\hsize
\setbox\onelinebox=\hbox{%
\vbox{\hbox{%
$\Pi_{20,59}$ spans $L_{251.3}$%
}\hbox{%
$32232232226223222622$%
}%
}%
\hfill\copy\matricesbox
}%
\ifdim\wd\onelinebox>\myboxwidth
\hbox to \myboxwidth{%
$\Pi_{20,59}$ spans $L_{251.3}$%
\hfil
$32232232226223222622$%
}%
\box\matricesbox
\else
\hbox to \myboxwidth{%
\unhbox\onelinebox
}%
\fi
\else
\hbox to \myboxwidth{%
$\Pi_{20,59}$ spans $L_{251.3}$%
\hfil}%
\hbox to \myboxwidth{%
$32232232226223222622$%
\hfil}%
\box\matricesbox
\fi
}%
\hfill\discretionary{}{}{}%

\vskip2pt\hrule\vskip2pt

\leavevmode\setbox\matricesbox=\hbox{%
{$\left[\!\llap{\phantom{%
\begingroup \smaller\smaller\smaller\begin{tabular}{@{}c@{}}%
\phantom{0}\\\phantom{0}\\\phantom{0}\\\phantom{0}
\end{tabular}\endgroup%
}}\right.$}%
\begingroup \smaller\smaller\smaller\begin{tabular}{@{}c@{}}%
-3\\\phantom{0}\\\phantom{0}\\\phantom{0}
\end{tabular}\endgroup%
\kern3pt%
\begingroup \smaller\smaller\smaller\begin{tabular}{@{}c@{}}%
\phantom{0}\\5\\\phantom{0}\\\phantom{0}
\end{tabular}\endgroup%
\kern3pt%
\begingroup \smaller\smaller\smaller\begin{tabular}{@{}c@{}}%
\phantom{0}\\\phantom{0}\\5\\\phantom{0}
\end{tabular}\endgroup%
\kern3pt%
\begingroup \smaller\smaller\smaller\begin{tabular}{@{}c@{}}%
\phantom{0}\\\phantom{0}\\\phantom{0}\\5
\end{tabular}\endgroup%
{$\left.\llap{\phantom{%
\begingroup \smaller\smaller\smaller\begin{tabular}{@{}c@{}}%
\phantom{0}\\\phantom{0}\\\phantom{0}\\\phantom{0}
\end{tabular}\endgroup%
}}\!\right]$}%
{$\left[\!\llap{\phantom{%
\begingroup \smaller\smaller\smaller\begin{tabular}{@{}c@{}}%
0\\0\\0\\0
\end{tabular}\endgroup%
}}\right.$}%
\begingroup \smaller\smaller\smaller\begin{tabular}{@{}c@{}}%
30\\19\\-8\\-11
\end{tabular}\endgroup%
\kern3pt%
\begingroup \smaller\smaller\smaller\begin{tabular}{@{}c@{}}%
10\\6\\-1\\-5
\end{tabular}\endgroup%
\kern3pt%
\begingroup \smaller\smaller\smaller\begin{tabular}{@{}c@{}}%
10\\5\\1\\-6
\end{tabular}\endgroup%
\kern3pt%
\begingroup \smaller\smaller\smaller\begin{tabular}{@{}c@{}}%
3\\1\\1\\-2
\end{tabular}\endgroup%
{$\left.\llap{\phantom{%
\begingroup \smaller\smaller\smaller\begin{tabular}{@{}c@{}}%
0\\0\\0\\0
\end{tabular}\endgroup%
}}\!\right]$}%
}%
\ifdim\wd\matricesbox>\halfwidth\myboxwidth=\hsize\else\myboxwidth=\halfwidth\fi
\vbox{%
\ifdim\myboxwidth=\hsize
\setbox\onelinebox=\hbox{%
\vbox{\hbox{%
$\Pi_{21,1}$ spans $L_{16.9}$%
}\hbox{%
$\slashthree632|236\slashthree632|236\slashthree632|236\rtimes D_{6}$%
}%
}%
\hfill\copy\matricesbox
}%
\ifdim\wd\onelinebox>\myboxwidth
\hbox to \myboxwidth{%
$\Pi_{21,1}$ spans $L_{16.9}$%
\hfil
$\slashthree632|236\slashthree632|236\slashthree632|236\rtimes D_{6}$%
}%
\box\matricesbox
\else
\hbox to \myboxwidth{%
\unhbox\onelinebox
}%
\fi
\else
\hbox to \myboxwidth{%
$\Pi_{21,1}$ spans $L_{16.9}$%
\hfil}%
\hbox to \myboxwidth{%
$\slashthree632|236\slashthree632|236\slashthree632|236\rtimes D_{6}$%
\hfil}%
\box\matricesbox
\fi
}%
\hfill\discretionary{}{}{}%
\setbox\matricesbox=\hbox{%
{$\left[\!\llap{\phantom{%
\begingroup \smaller\smaller\smaller\begin{tabular}{@{}c@{}}%
\phantom{0}\\\phantom{0}\\\phantom{0}\\\phantom{0}
\end{tabular}\endgroup%
}}\right.$}%
\begingroup \smaller\smaller\smaller\begin{tabular}{@{}c@{}}%
-5\\\phantom{0}\\\phantom{0}\\\phantom{0}
\end{tabular}\endgroup%
\kern3pt%
\begingroup \smaller\smaller\smaller\begin{tabular}{@{}c@{}}%
\phantom{0}\\1\\\phantom{0}\\\phantom{0}
\end{tabular}\endgroup%
\kern3pt%
\begingroup \smaller\smaller\smaller\begin{tabular}{@{}c@{}}%
\phantom{0}\\\phantom{0}\\1\\\phantom{0}
\end{tabular}\endgroup%
\kern3pt%
\begingroup \smaller\smaller\smaller\begin{tabular}{@{}c@{}}%
\phantom{0}\\\phantom{0}\\\phantom{0}\\1
\end{tabular}\endgroup%
{$\left.\llap{\phantom{%
\begingroup \smaller\smaller\smaller\begin{tabular}{@{}c@{}}%
\phantom{0}\\\phantom{0}\\\phantom{0}\\\phantom{0}
\end{tabular}\endgroup%
}}\!\right]$}%
{$\left[\!\llap{\phantom{%
\begingroup \smaller\smaller\smaller\begin{tabular}{@{}c@{}}%
0\\0\\0\\0
\end{tabular}\endgroup%
}}\right.$}%
\begingroup \smaller\smaller\smaller\begin{tabular}{@{}c@{}}%
6\\-11\\4\\7
\end{tabular}\endgroup%
\kern3pt%
\begingroup \smaller\smaller\smaller\begin{tabular}{@{}c@{}}%
18\\-30\\3\\27
\end{tabular}\endgroup%
\kern3pt%
\begingroup \smaller\smaller\smaller\begin{tabular}{@{}c@{}}%
18\\-27\\-3\\30
\end{tabular}\endgroup%
\kern3pt%
\begingroup \smaller\smaller\smaller\begin{tabular}{@{}c@{}}%
1\\-1\\-1\\2
\end{tabular}\endgroup%
{$\left.\llap{\phantom{%
\begingroup \smaller\smaller\smaller\begin{tabular}{@{}c@{}}%
0\\0\\0\\0
\end{tabular}\endgroup%
}}\!\right]$}%
}%
\ifdim\wd\matricesbox>\halfwidth\myboxwidth=\hsize\else\myboxwidth=\halfwidth\fi
\vbox{%
\ifdim\myboxwidth=\hsize
\setbox\onelinebox=\hbox{%
\vbox{\hbox{%
$\Pi_{21,2}$ spans $L_{16.7}$%
}\hbox{%
$36\slashthree632|236\slashthree632|236\slashthree632|2\rtimes D_{6}$%
}%
}%
\hfill\copy\matricesbox
}%
\ifdim\wd\onelinebox>\myboxwidth
\hbox to \myboxwidth{%
$\Pi_{21,2}$ spans $L_{16.7}$%
\hfil
$36\slashthree632|236\slashthree632|236\slashthree632|2\rtimes D_{6}$%
}%
\box\matricesbox
\else
\hbox to \myboxwidth{%
\unhbox\onelinebox
}%
\fi
\else
\hbox to \myboxwidth{%
$\Pi_{21,2}$ spans $L_{16.7}$%
\hfil}%
\hbox to \myboxwidth{%
$36\slashthree632|236\slashthree632|236\slashthree632|2\rtimes D_{6}$%
\hfil}%
\box\matricesbox
\fi
}%
\hfill\discretionary{}{}{}%
\setbox\matricesbox=\hbox{%
{$\left[\!\llap{\phantom{%
\begingroup \smaller\smaller\smaller\begin{tabular}{@{}c@{}}%
\phantom{0}\\\phantom{0}\\\phantom{0}\\\phantom{0}
\end{tabular}\endgroup%
}}\right.$}%
\begingroup \smaller\smaller\smaller\begin{tabular}{@{}c@{}}%
-5\\\phantom{0}\\\phantom{0}\\\phantom{0}
\end{tabular}\endgroup%
\kern3pt%
\begingroup \smaller\smaller\smaller\begin{tabular}{@{}c@{}}%
\phantom{0}\\6\\\phantom{0}\\\phantom{0}
\end{tabular}\endgroup%
\kern3pt%
\begingroup \smaller\smaller\smaller\begin{tabular}{@{}c@{}}%
\phantom{0}\\\phantom{0}\\6\\\phantom{0}
\end{tabular}\endgroup%
\kern3pt%
\begingroup \smaller\smaller\smaller\begin{tabular}{@{}c@{}}%
\phantom{0}\\\phantom{0}\\\phantom{0}\\6
\end{tabular}\endgroup%
{$\left.\llap{\phantom{%
\begingroup \smaller\smaller\smaller\begin{tabular}{@{}c@{}}%
\phantom{0}\\\phantom{0}\\\phantom{0}\\\phantom{0}
\end{tabular}\endgroup%
}}\!\right]$}%
{$\left[\!\llap{\phantom{%
\begingroup \smaller\smaller\smaller\begin{tabular}{@{}c@{}}%
0\\0\\0\\0
\end{tabular}\endgroup%
}}\right.$}%
\begingroup \smaller\smaller\smaller\begin{tabular}{@{}c@{}}%
4\\3\\-2\\-1
\end{tabular}\endgroup%
\kern3pt%
\begingroup \smaller\smaller\smaller\begin{tabular}{@{}c@{}}%
4\\3\\-1\\-2
\end{tabular}\endgroup%
\kern3pt%
\begingroup \smaller\smaller\smaller\begin{tabular}{@{}c@{}}%
3\\2\\0\\-2
\end{tabular}\endgroup%
\kern3pt%
\begingroup \smaller\smaller\smaller\begin{tabular}{@{}c@{}}%
4\\2\\1\\-3
\end{tabular}\endgroup%
\kern3pt%
\begingroup \smaller\smaller\smaller\begin{tabular}{@{}c@{}}%
12\\4\\5\\-9
\end{tabular}\endgroup%
\kern3pt%
\begingroup \smaller\smaller\smaller\begin{tabular}{@{}c@{}}%
36\\7\\19\\-26
\end{tabular}\endgroup%
\kern3pt%
\begingroup \smaller\smaller\smaller\begin{tabular}{@{}c@{}}%
3\\0\\2\\-2
\end{tabular}\endgroup%
{$\left.\llap{\phantom{%
\begingroup \smaller\smaller\smaller\begin{tabular}{@{}c@{}}%
0\\0\\0\\0
\end{tabular}\endgroup%
}}\!\right]$}%
}%
\ifdim\wd\matricesbox>\halfwidth\myboxwidth=\hsize\else\myboxwidth=\halfwidth\fi
\vbox{%
\ifdim\myboxwidth=\hsize
\setbox\onelinebox=\hbox{%
\vbox{\hbox{%
$\Pi_{21,3}$ spans $L_{251.3}$%
}\hbox{%
$322262232226223222622\rtimes C_{3}$%
}%
}%
\hfill\copy\matricesbox
}%
\ifdim\wd\onelinebox>\myboxwidth
\hbox to \myboxwidth{%
$\Pi_{21,3}$ spans $L_{251.3}$%
\hfil
$322262232226223222622\rtimes C_{3}$%
}%
\box\matricesbox
\else
\hbox to \myboxwidth{%
\unhbox\onelinebox
}%
\fi
\else
\hbox to \myboxwidth{%
$\Pi_{21,3}$ spans $L_{251.3}$%
\hfil}%
\hbox to \myboxwidth{%
$322262232226223222622\rtimes C_{3}$%
\hfil}%
\box\matricesbox
\fi
}%
\hfill\discretionary{}{}{}%
\setbox\matricesbox=\hbox{%
{$\left[\!\llap{\phantom{%
\begingroup \smaller\smaller\smaller
\endgroup%
}}\!\right]$}%
}%
\ifdim\wd\matricesbox>\halfwidth\myboxwidth=\hsize\else\myboxwidth=\halfwidth\fi
\vbox{%
\ifdim\myboxwidth=\hsize
\setbox\onelinebox=\hbox{%
\vbox{\hbox{%
$\Pi_{21,4}$ spans $L_{16.13}$%
}\hbox{%
$36322|2236363636\slashthree63636\rtimes D_{2}$%
}%
}%
\hfill\copy\matricesbox
}%
\ifdim\wd\onelinebox>\myboxwidth
\hbox to \myboxwidth{%
$\Pi_{21,4}$ spans $L_{16.13}$%
\hfil
$36322|2236363636\slashthree63636\rtimes D_{2}$%
}%
\box\matricesbox
\else
\hbox to \myboxwidth{%
\unhbox\onelinebox
}%
\fi
\else
\hbox to \myboxwidth{%
$\Pi_{21,4}$ spans $L_{16.13}$%
\hfil}%
\hbox to \myboxwidth{%
$36322|2236363636\slashthree63636\rtimes D_{2}$%
\hfil}%
\box\matricesbox
\fi
}%
\hfill\discretionary{}{}{}%
\setbox\matricesbox=\hbox{%
{$\left[\!\llap{\phantom{%
\begingroup \smaller\smaller\smaller
\endgroup%
}}\!\right]$}%
}%
\ifdim\wd\matricesbox>\halfwidth\myboxwidth=\hsize\else\myboxwidth=\halfwidth\fi
\vbox{%
\ifdim\myboxwidth=\hsize
\setbox\onelinebox=\hbox{%
\vbox{\hbox{%
$\Pi_{21,5}$ spans $L_{16.9}$%
}\hbox{%
$3632232|2322363636\slashthree636\rtimes D_{2}$%
}%
}%
\hfill\copy\matricesbox
}%
\ifdim\wd\onelinebox>\myboxwidth
\hbox to \myboxwidth{%
$\Pi_{21,5}$ spans $L_{16.9}$%
\hfil
$3632232|2322363636\slashthree636\rtimes D_{2}$%
}%
\box\matricesbox
\else
\hbox to \myboxwidth{%
\unhbox\onelinebox
}%
\fi
\else
\hbox to \myboxwidth{%
$\Pi_{21,5}$ spans $L_{16.9}$%
\hfil}%
\hbox to \myboxwidth{%
$3632232|2322363636\slashthree636\rtimes D_{2}$%
\hfil}%
\box\matricesbox
\fi
}%
\hfill\discretionary{}{}{}%
\setbox\matricesbox=\hbox{%
{$\left[\!\llap{\phantom{%
\begingroup \smaller\smaller\smaller
\endgroup%
}}\!\right]$}%
}%
\ifdim\wd\matricesbox>\halfwidth\myboxwidth=\hsize\else\myboxwidth=\halfwidth\fi
\vbox{%
\ifdim\myboxwidth=\hsize
\setbox\onelinebox=\hbox{%
\vbox{\hbox{%
$\Pi_{21,6}$ spans $L_{16.13}$%
}\hbox{%
$36322\slashthree2236363632|23636\rtimes D_{2}$%
}%
}%
\hfill\copy\matricesbox
}%
\ifdim\wd\onelinebox>\myboxwidth
\hbox to \myboxwidth{%
$\Pi_{21,6}$ spans $L_{16.13}$%
\hfil
$36322\slashthree2236363632|23636\rtimes D_{2}$%
}%
\box\matricesbox
\else
\hbox to \myboxwidth{%
\unhbox\onelinebox
}%
\fi
\else
\hbox to \myboxwidth{%
$\Pi_{21,6}$ spans $L_{16.13}$%
\hfil}%
\hbox to \myboxwidth{%
$36322\slashthree2236363632|23636\rtimes D_{2}$%
\hfil}%
\box\matricesbox
\fi
}%
\hfill\discretionary{}{}{}%
\setbox\matricesbox=\hbox{%
{$\left[\!\llap{\phantom{%
\begingroup \smaller\smaller\smaller
\endgroup%
}}\!\right]$}%
}%
\ifdim\wd\matricesbox>\halfwidth\myboxwidth=\hsize\else\myboxwidth=\halfwidth\fi
\vbox{%
\ifdim\myboxwidth=\hsize
\setbox\onelinebox=\hbox{%
\vbox{\hbox{%
$\Pi_{21,7}$ spans $L_{16.13}$%
}\hbox{%
$363223632|2363223636\slashthree6\rtimes D_{2}$%
}%
}%
\hfill\copy\matricesbox
}%
\ifdim\wd\onelinebox>\myboxwidth
\hbox to \myboxwidth{%
$\Pi_{21,7}$ spans $L_{16.13}$%
\hfil
$363223632|2363223636\slashthree6\rtimes D_{2}$%
}%
\box\matricesbox
\else
\hbox to \myboxwidth{%
\unhbox\onelinebox
}%
\fi
\else
\hbox to \myboxwidth{%
$\Pi_{21,7}$ spans $L_{16.13}$%
\hfil}%
\hbox to \myboxwidth{%
$363223632|2363223636\slashthree6\rtimes D_{2}$%
\hfil}%
\box\matricesbox
\fi
}%
\hfill\discretionary{}{}{}%
\setbox\matricesbox=\hbox{%
{$\left[\!\llap{\phantom{%
\begingroup \smaller\smaller\smaller
\endgroup%
}}\!\right]$}%
}%
\ifdim\wd\matricesbox>\halfwidth\myboxwidth=\hsize\else\myboxwidth=\halfwidth\fi
\vbox{%
\ifdim\myboxwidth=\hsize
\setbox\onelinebox=\hbox{%
\vbox{\hbox{%
$\Pi_{21,8}$ spans $L_{16.13}$%
}\hbox{%
$3632|2363223636\slashthree636322\rtimes D_{2}$%
}%
}%
\hfill\copy\matricesbox
}%
\ifdim\wd\onelinebox>\myboxwidth
\hbox to \myboxwidth{%
$\Pi_{21,8}$ spans $L_{16.13}$%
\hfil
$3632|2363223636\slashthree636322\rtimes D_{2}$%
}%
\box\matricesbox
\else
\hbox to \myboxwidth{%
\unhbox\onelinebox
}%
\fi
\else
\hbox to \myboxwidth{%
$\Pi_{21,8}$ spans $L_{16.13}$%
\hfil}%
\hbox to \myboxwidth{%
$3632|2363223636\slashthree636322\rtimes D_{2}$%
\hfil}%
\box\matricesbox
\fi
}%
\hfill\discretionary{}{}{}%
\setbox\matricesbox=\hbox{%
{$\left[\!\llap{\phantom{%
\begingroup \smaller\smaller\smaller
\endgroup%
}}\!\right]$}%
}%
\ifdim\wd\matricesbox>\halfwidth\myboxwidth=\hsize\else\myboxwidth=\halfwidth\fi
\vbox{%
\ifdim\myboxwidth=\hsize
\setbox\onelinebox=\hbox{%
\vbox{\hbox{%
$\Pi_{21,9}$ spans $L_{16.13}$%
}\hbox{%
$3632|2363636322\slashthree223636\rtimes D_{2}$%
}%
}%
\hfill\copy\matricesbox
}%
\ifdim\wd\onelinebox>\myboxwidth
\hbox to \myboxwidth{%
$\Pi_{21,9}$ spans $L_{16.13}$%
\hfil
$3632|2363636322\slashthree223636\rtimes D_{2}$%
}%
\box\matricesbox
\else
\hbox to \myboxwidth{%
\unhbox\onelinebox
}%
\fi
\else
\hbox to \myboxwidth{%
$\Pi_{21,9}$ spans $L_{16.13}$%
\hfil}%
\hbox to \myboxwidth{%
$3632|2363636322\slashthree223636\rtimes D_{2}$%
\hfil}%
\box\matricesbox
\fi
}%
\hfill\discretionary{}{}{}%
\setbox\matricesbox=\hbox{%
{$\left[\!\llap{\phantom{%
\begingroup \smaller\smaller\smaller
\endgroup%
}}\!\right]$}%
}%
\ifdim\wd\matricesbox>\halfwidth\myboxwidth=\hsize\else\myboxwidth=\halfwidth\fi
\vbox{%
\ifdim\myboxwidth=\hsize
\setbox\onelinebox=\hbox{%
\vbox{\hbox{%
$\Pi_{21,10}$ spans $L_{16.13}$%
}\hbox{%
$3636322|2236363636\slashthree636\rtimes D_{2}$%
}%
}%
\hfill\copy\matricesbox
}%
\ifdim\wd\onelinebox>\myboxwidth
\hbox to \myboxwidth{%
$\Pi_{21,10}$ spans $L_{16.13}$%
\hfil
$3636322|2236363636\slashthree636\rtimes D_{2}$%
}%
\box\matricesbox
\else
\hbox to \myboxwidth{%
\unhbox\onelinebox
}%
\fi
\else
\hbox to \myboxwidth{%
$\Pi_{21,10}$ spans $L_{16.13}$%
\hfil}%
\hbox to \myboxwidth{%
$3636322|2236363636\slashthree636\rtimes D_{2}$%
\hfil}%
\box\matricesbox
\fi
}%
\hfill\discretionary{}{}{}%
\setbox\matricesbox=\hbox{%
{$\left[\!\llap{\phantom{%
\begingroup \smaller\smaller\smaller
\endgroup%
}}\!\right]$}%
}%
\ifdim\wd\matricesbox>\halfwidth\myboxwidth=\hsize\else\myboxwidth=\halfwidth\fi
\vbox{%
\ifdim\myboxwidth=\hsize
\setbox\onelinebox=\hbox{%
\vbox{\hbox{%
$\Pi_{21,11}$ spans $L_{16.7}$%
}\hbox{%
$363632232|2322363636\slashthree6\rtimes D_{2}$%
}%
}%
\hfill\copy\matricesbox
}%
\ifdim\wd\onelinebox>\myboxwidth
\hbox to \myboxwidth{%
$\Pi_{21,11}$ spans $L_{16.7}$%
\hfil
$363632232|2322363636\slashthree6\rtimes D_{2}$%
}%
\box\matricesbox
\else
\hbox to \myboxwidth{%
\unhbox\onelinebox
}%
\fi
\else
\hbox to \myboxwidth{%
$\Pi_{21,11}$ spans $L_{16.7}$%
\hfil}%
\hbox to \myboxwidth{%
$363632232|2322363636\slashthree6\rtimes D_{2}$%
\hfil}%
\box\matricesbox
\fi
}%
\hfill\discretionary{}{}{}%
\setbox\matricesbox=\hbox{%
{$\left[\!\llap{\phantom{%
\begingroup \smaller\smaller\smaller
\endgroup%
}}\!\right]$}%
}%
\ifdim\wd\matricesbox>\halfwidth\myboxwidth=\hsize\else\myboxwidth=\halfwidth\fi
\vbox{%
\ifdim\myboxwidth=\hsize
\setbox\onelinebox=\hbox{%
\vbox{\hbox{%
$\Pi_{21,12}$ spans $L_{16.13}$%
}\hbox{%
$363222322363636363636$%
}%
}%
\hfill\copy\matricesbox
}%
\ifdim\wd\onelinebox>\myboxwidth
\hbox to \myboxwidth{%
$\Pi_{21,12}$ spans $L_{16.13}$%
\hfil
$363222322363636363636$%
}%
\box\matricesbox
\else
\hbox to \myboxwidth{%
\unhbox\onelinebox
}%
\fi
\else
\hbox to \myboxwidth{%
$\Pi_{21,12}$ spans $L_{16.13}$%
\hfil}%
\hbox to \myboxwidth{%
$363222322363636363636$%
\hfil}%
\box\matricesbox
\fi
}%
\hfill\discretionary{}{}{}%
\setbox\matricesbox=\hbox{%
{$\left[\!\llap{\phantom{%
\begingroup \smaller\smaller\smaller
\endgroup%
}}\!\right]$}%
}%
\ifdim\wd\matricesbox>\halfwidth\myboxwidth=\hsize\else\myboxwidth=\halfwidth\fi
\vbox{%
\ifdim\myboxwidth=\hsize
\setbox\onelinebox=\hbox{%
\vbox{\hbox{%
$\Pi_{21,13}$ spans $L_{16.13}$%
}\hbox{%
$363222363223636363636$%
}%
}%
\hfill\copy\matricesbox
}%
\ifdim\wd\onelinebox>\myboxwidth
\hbox to \myboxwidth{%
$\Pi_{21,13}$ spans $L_{16.13}$%
\hfil
$363222363223636363636$%
}%
\box\matricesbox
\else
\hbox to \myboxwidth{%
\unhbox\onelinebox
}%
\fi
\else
\hbox to \myboxwidth{%
$\Pi_{21,13}$ spans $L_{16.13}$%
\hfil}%
\hbox to \myboxwidth{%
$363222363223636363636$%
\hfil}%
\box\matricesbox
\fi
}%
\hfill\discretionary{}{}{}%
\setbox\matricesbox=\hbox{%
{$\left[\!\llap{\phantom{%
\begingroup \smaller\smaller\smaller
\endgroup%
}}\!\right]$}%
}%
\ifdim\wd\matricesbox>\halfwidth\myboxwidth=\hsize\else\myboxwidth=\halfwidth\fi
\vbox{%
\ifdim\myboxwidth=\hsize
\setbox\onelinebox=\hbox{%
\vbox{\hbox{%
$\Pi_{21,14}$ spans $L_{16.13}$%
}\hbox{%
$363222363632236363636$%
}%
}%
\hfill\copy\matricesbox
}%
\ifdim\wd\onelinebox>\myboxwidth
\hbox to \myboxwidth{%
$\Pi_{21,14}$ spans $L_{16.13}$%
\hfil
$363222363632236363636$%
}%
\box\matricesbox
\else
\hbox to \myboxwidth{%
\unhbox\onelinebox
}%
\fi
\else
\hbox to \myboxwidth{%
$\Pi_{21,14}$ spans $L_{16.13}$%
\hfil}%
\hbox to \myboxwidth{%
$363222363632236363636$%
\hfil}%
\box\matricesbox
\fi
}%
\hfill\discretionary{}{}{}%
\setbox\matricesbox=\hbox{%
{$\left[\!\llap{\phantom{%
\begingroup \smaller\smaller\smaller
\endgroup%
}}\!\right]$}%
}%
\ifdim\wd\matricesbox>\halfwidth\myboxwidth=\hsize\else\myboxwidth=\halfwidth\fi
\vbox{%
\ifdim\myboxwidth=\hsize
\setbox\onelinebox=\hbox{%
\vbox{\hbox{%
$\Pi_{21,15}$ spans $L_{16.13}$%
}\hbox{%
$363222363636322363636$%
}%
}%
\hfill\copy\matricesbox
}%
\ifdim\wd\onelinebox>\myboxwidth
\hbox to \myboxwidth{%
$\Pi_{21,15}$ spans $L_{16.13}$%
\hfil
$363222363636322363636$%
}%
\box\matricesbox
\else
\hbox to \myboxwidth{%
\unhbox\onelinebox
}%
\fi
\else
\hbox to \myboxwidth{%
$\Pi_{21,15}$ spans $L_{16.13}$%
\hfil}%
\hbox to \myboxwidth{%
$363222363636322363636$%
\hfil}%
\box\matricesbox
\fi
}%
\hfill\discretionary{}{}{}%
\setbox\matricesbox=\hbox{%
{$\left[\!\llap{\phantom{%
\begingroup \smaller\smaller\smaller
\endgroup%
}}\!\right]$}%
}%
\ifdim\wd\matricesbox>\halfwidth\myboxwidth=\hsize\else\myboxwidth=\halfwidth\fi
\vbox{%
\ifdim\myboxwidth=\hsize
\setbox\onelinebox=\hbox{%
\vbox{\hbox{%
$\Pi_{21,16}$ spans $L_{16.13}$%
}\hbox{%
$363222363636363223636$%
}%
}%
\hfill\copy\matricesbox
}%
\ifdim\wd\onelinebox>\myboxwidth
\hbox to \myboxwidth{%
$\Pi_{21,16}$ spans $L_{16.13}$%
\hfil
$363222363636363223636$%
}%
\box\matricesbox
\else
\hbox to \myboxwidth{%
\unhbox\onelinebox
}%
\fi
\else
\hbox to \myboxwidth{%
$\Pi_{21,16}$ spans $L_{16.13}$%
\hfil}%
\hbox to \myboxwidth{%
$363222363636363223636$%
\hfil}%
\box\matricesbox
\fi
}%
\hfill\discretionary{}{}{}%
\setbox\matricesbox=\hbox{%
{$\left[\!\llap{\phantom{%
\begingroup \smaller\smaller\smaller
\endgroup%
}}\!\right]$}%
}%
\ifdim\wd\matricesbox>\halfwidth\myboxwidth=\hsize\else\myboxwidth=\halfwidth\fi
\vbox{%
\ifdim\myboxwidth=\hsize
\setbox\onelinebox=\hbox{%
\vbox{\hbox{%
$\Pi_{21,17}$ spans $L_{16.13}$%
}\hbox{%
$363222363636363632236$%
}%
}%
\hfill\copy\matricesbox
}%
\ifdim\wd\onelinebox>\myboxwidth
\hbox to \myboxwidth{%
$\Pi_{21,17}$ spans $L_{16.13}$%
\hfil
$363222363636363632236$%
}%
\box\matricesbox
\else
\hbox to \myboxwidth{%
\unhbox\onelinebox
}%
\fi
\else
\hbox to \myboxwidth{%
$\Pi_{21,17}$ spans $L_{16.13}$%
\hfil}%
\hbox to \myboxwidth{%
$363222363636363632236$%
\hfil}%
\box\matricesbox
\fi
}%
\hfill\discretionary{}{}{}%
\setbox\matricesbox=\hbox{%
{$\left[\!\llap{\phantom{%
\begingroup \smaller\smaller\smaller
\endgroup%
}}\!\right]$}%
}%
\ifdim\wd\matricesbox>\halfwidth\myboxwidth=\hsize\else\myboxwidth=\halfwidth\fi
\vbox{%
\ifdim\myboxwidth=\hsize
\setbox\onelinebox=\hbox{%
\vbox{\hbox{%
$\Pi_{21,18}$ spans $L_{16.13}$%
}\hbox{%
$363222363636363636322$%
}%
}%
\hfill\copy\matricesbox
}%
\ifdim\wd\onelinebox>\myboxwidth
\hbox to \myboxwidth{%
$\Pi_{21,18}$ spans $L_{16.13}$%
\hfil
$363222363636363636322$%
}%
\box\matricesbox
\else
\hbox to \myboxwidth{%
\unhbox\onelinebox
}%
\fi
\else
\hbox to \myboxwidth{%
$\Pi_{21,18}$ spans $L_{16.13}$%
\hfil}%
\hbox to \myboxwidth{%
$363222363636363636322$%
\hfil}%
\box\matricesbox
\fi
}%
\hfill\discretionary{}{}{}%
\setbox\matricesbox=\hbox{%
{$\left[\!\llap{\phantom{%
\begingroup \smaller\smaller\smaller
\endgroup%
}}\!\right]$}%
}%
\ifdim\wd\matricesbox>\halfwidth\myboxwidth=\hsize\else\myboxwidth=\halfwidth\fi
\vbox{%
\ifdim\myboxwidth=\hsize
\setbox\onelinebox=\hbox{%
\vbox{\hbox{%
$\Pi_{21,19}$ spans $L_{16.13}$%
}\hbox{%
$363223222363636363636$%
}%
}%
\hfill\copy\matricesbox
}%
\ifdim\wd\onelinebox>\myboxwidth
\hbox to \myboxwidth{%
$\Pi_{21,19}$ spans $L_{16.13}$%
\hfil
$363223222363636363636$%
}%
\box\matricesbox
\else
\hbox to \myboxwidth{%
\unhbox\onelinebox
}%
\fi
\else
\hbox to \myboxwidth{%
$\Pi_{21,19}$ spans $L_{16.13}$%
\hfil}%
\hbox to \myboxwidth{%
$363223222363636363636$%
\hfil}%
\box\matricesbox
\fi
}%
\hfill\discretionary{}{}{}%
\setbox\matricesbox=\hbox{%
{$\left[\!\llap{\phantom{%
\begingroup \smaller\smaller\smaller
\endgroup%
}}\!\right]$}%
}%
\ifdim\wd\matricesbox>\halfwidth\myboxwidth=\hsize\else\myboxwidth=\halfwidth\fi
\vbox{%
\ifdim\myboxwidth=\hsize
\setbox\onelinebox=\hbox{%
\vbox{\hbox{%
$\Pi_{21,20}$ spans $L_{16.13}$%
}\hbox{%
$363223223632236363636$%
}%
}%
\hfill\copy\matricesbox
}%
\ifdim\wd\onelinebox>\myboxwidth
\hbox to \myboxwidth{%
$\Pi_{21,20}$ spans $L_{16.13}$%
\hfil
$363223223632236363636$%
}%
\box\matricesbox
\else
\hbox to \myboxwidth{%
\unhbox\onelinebox
}%
\fi
\else
\hbox to \myboxwidth{%
$\Pi_{21,20}$ spans $L_{16.13}$%
\hfil}%
\hbox to \myboxwidth{%
$363223223632236363636$%
\hfil}%
\box\matricesbox
\fi
}%
\hfill\discretionary{}{}{}%
\setbox\matricesbox=\hbox{%
{$\left[\!\llap{\phantom{%
\begingroup \smaller\smaller\smaller
\endgroup%
}}\!\right]$}%
}%
\ifdim\wd\matricesbox>\halfwidth\myboxwidth=\hsize\else\myboxwidth=\halfwidth\fi
\vbox{%
\ifdim\myboxwidth=\hsize
\setbox\onelinebox=\hbox{%
\vbox{\hbox{%
$\Pi_{21,21}$ spans $L_{16.9}$%
}\hbox{%
$363223223636322363636$%
}%
}%
\hfill\copy\matricesbox
}%
\ifdim\wd\onelinebox>\myboxwidth
\hbox to \myboxwidth{%
$\Pi_{21,21}$ spans $L_{16.9}$%
\hfil
$363223223636322363636$%
}%
\box\matricesbox
\else
\hbox to \myboxwidth{%
\unhbox\onelinebox
}%
\fi
\else
\hbox to \myboxwidth{%
$\Pi_{21,21}$ spans $L_{16.9}$%
\hfil}%
\hbox to \myboxwidth{%
$363223223636322363636$%
\hfil}%
\box\matricesbox
\fi
}%
\hfill\discretionary{}{}{}%
\setbox\matricesbox=\hbox{%
{$\left[\!\llap{\phantom{%
\begingroup \smaller\smaller\smaller
\endgroup%
}}\!\right]$}%
}%
\ifdim\wd\matricesbox>\halfwidth\myboxwidth=\hsize\else\myboxwidth=\halfwidth\fi
\vbox{%
\ifdim\myboxwidth=\hsize
\setbox\onelinebox=\hbox{%
\vbox{\hbox{%
$\Pi_{21,22}$ spans $L_{16.13}$%
}\hbox{%
$363223632232236363636$%
}%
}%
\hfill\copy\matricesbox
}%
\ifdim\wd\onelinebox>\myboxwidth
\hbox to \myboxwidth{%
$\Pi_{21,22}$ spans $L_{16.13}$%
\hfil
$363223632232236363636$%
}%
\box\matricesbox
\else
\hbox to \myboxwidth{%
\unhbox\onelinebox
}%
\fi
\else
\hbox to \myboxwidth{%
$\Pi_{21,22}$ spans $L_{16.13}$%
\hfil}%
\hbox to \myboxwidth{%
$363223632232236363636$%
\hfil}%
\box\matricesbox
\fi
}%
\hfill\discretionary{}{}{}%
\setbox\matricesbox=\hbox{%
{$\left[\!\llap{\phantom{%
\begingroup \smaller\smaller\smaller
\endgroup%
}}\!\right]$}%
}%
\ifdim\wd\matricesbox>\halfwidth\myboxwidth=\hsize\else\myboxwidth=\halfwidth\fi
\vbox{%
\ifdim\myboxwidth=\hsize
\setbox\onelinebox=\hbox{%
\vbox{\hbox{%
$\Pi_{21,23}$ spans $L_{16.13}$%
}\hbox{%
$363223632236363223636$%
}%
}%
\hfill\copy\matricesbox
}%
\ifdim\wd\onelinebox>\myboxwidth
\hbox to \myboxwidth{%
$\Pi_{21,23}$ spans $L_{16.13}$%
\hfil
$363223632236363223636$%
}%
\box\matricesbox
\else
\hbox to \myboxwidth{%
\unhbox\onelinebox
}%
\fi
\else
\hbox to \myboxwidth{%
$\Pi_{21,23}$ spans $L_{16.13}$%
\hfil}%
\hbox to \myboxwidth{%
$363223632236363223636$%
\hfil}%
\box\matricesbox
\fi
}%
\hfill\discretionary{}{}{}%
\setbox\matricesbox=\hbox{%
{$\left[\!\llap{\phantom{%
\begingroup \smaller\smaller\smaller
\endgroup%
}}\!\right]$}%
}%
\ifdim\wd\matricesbox>\halfwidth\myboxwidth=\hsize\else\myboxwidth=\halfwidth\fi
\vbox{%
\ifdim\myboxwidth=\hsize
\setbox\onelinebox=\hbox{%
\vbox{\hbox{%
$\Pi_{21,24}$ spans $L_{16.13}$%
}\hbox{%
$363223632236363632236$%
}%
}%
\hfill\copy\matricesbox
}%
\ifdim\wd\onelinebox>\myboxwidth
\hbox to \myboxwidth{%
$\Pi_{21,24}$ spans $L_{16.13}$%
\hfil
$363223632236363632236$%
}%
\box\matricesbox
\else
\hbox to \myboxwidth{%
\unhbox\onelinebox
}%
\fi
\else
\hbox to \myboxwidth{%
$\Pi_{21,24}$ spans $L_{16.13}$%
\hfil}%
\hbox to \myboxwidth{%
$363223632236363632236$%
\hfil}%
\box\matricesbox
\fi
}%
\hfill\discretionary{}{}{}%
\setbox\matricesbox=\hbox{%
{$\left[\!\llap{\phantom{%
\begingroup \smaller\smaller\smaller
\endgroup%
}}\!\right]$}%
}%
\ifdim\wd\matricesbox>\halfwidth\myboxwidth=\hsize\else\myboxwidth=\halfwidth\fi
\vbox{%
\ifdim\myboxwidth=\hsize
\setbox\onelinebox=\hbox{%
\vbox{\hbox{%
$\Pi_{21,25}$ spans $L_{16.7}$%
}\hbox{%
$363632232236363223636$%
}%
}%
\hfill\copy\matricesbox
}%
\ifdim\wd\onelinebox>\myboxwidth
\hbox to \myboxwidth{%
$\Pi_{21,25}$ spans $L_{16.7}$%
\hfil
$363632232236363223636$%
}%
\box\matricesbox
\else
\hbox to \myboxwidth{%
\unhbox\onelinebox
}%
\fi
\else
\hbox to \myboxwidth{%
$\Pi_{21,25}$ spans $L_{16.7}$%
\hfil}%
\hbox to \myboxwidth{%
$363632232236363223636$%
\hfil}%
\box\matricesbox
\fi
}%
\hfill\discretionary{}{}{}%
\setbox\matricesbox=\hbox{%
{$\left[\!\llap{\phantom{%
\begingroup \smaller\smaller\smaller
\endgroup%
}}\!\right]$}%
}%
\ifdim\wd\matricesbox>\halfwidth\myboxwidth=\hsize\else\myboxwidth=\halfwidth\fi
\vbox{%
\ifdim\myboxwidth=\hsize
\setbox\onelinebox=\hbox{%
\vbox{\hbox{%
$\Pi_{21,26}$ spans $L_{251.3}$%
}\hbox{%
$322322322262226222622$%
}%
}%
\hfill\copy\matricesbox
}%
\ifdim\wd\onelinebox>\myboxwidth
\hbox to \myboxwidth{%
$\Pi_{21,26}$ spans $L_{251.3}$%
\hfil
$322322322262226222622$%
}%
\box\matricesbox
\else
\hbox to \myboxwidth{%
\unhbox\onelinebox
}%
\fi
\else
\hbox to \myboxwidth{%
$\Pi_{21,26}$ spans $L_{251.3}$%
\hfil}%
\hbox to \myboxwidth{%
$322322322262226222622$%
\hfil}%
\box\matricesbox
\fi
}%
\hfill\discretionary{}{}{}%
\setbox\matricesbox=\hbox{%
{$\left[\!\llap{\phantom{%
\begingroup \smaller\smaller\smaller
\endgroup%
}}\!\right]$}%
}%
\ifdim\wd\matricesbox>\halfwidth\myboxwidth=\hsize\else\myboxwidth=\halfwidth\fi
\vbox{%
\ifdim\myboxwidth=\hsize
\setbox\onelinebox=\hbox{%
\vbox{\hbox{%
$\Pi_{21,27}$ spans $L_{251.3}$%
}\hbox{%
$322322322262226226222$%
}%
}%
\hfill\copy\matricesbox
}%
\ifdim\wd\onelinebox>\myboxwidth
\hbox to \myboxwidth{%
$\Pi_{21,27}$ spans $L_{251.3}$%
\hfil
$322322322262226226222$%
}%
\box\matricesbox
\else
\hbox to \myboxwidth{%
\unhbox\onelinebox
}%
\fi
\else
\hbox to \myboxwidth{%
$\Pi_{21,27}$ spans $L_{251.3}$%
\hfil}%
\hbox to \myboxwidth{%
$322322322262226226222$%
\hfil}%
\box\matricesbox
\fi
}%
\hfill\discretionary{}{}{}%
\setbox\matricesbox=\hbox{%
{$\left[\!\llap{\phantom{%
\begingroup \smaller\smaller\smaller
\endgroup%
}}\!\right]$}%
}%
\ifdim\wd\matricesbox>\halfwidth\myboxwidth=\hsize\else\myboxwidth=\halfwidth\fi
\vbox{%
\ifdim\myboxwidth=\hsize
\setbox\onelinebox=\hbox{%
\vbox{\hbox{%
$\Pi_{21,28}$ spans $L_{251.3}$%
}\hbox{%
$322322322262262222622$%
}%
}%
\hfill\copy\matricesbox
}%
\ifdim\wd\onelinebox>\myboxwidth
\hbox to \myboxwidth{%
$\Pi_{21,28}$ spans $L_{251.3}$%
\hfil
$322322322262262222622$%
}%
\box\matricesbox
\else
\hbox to \myboxwidth{%
\unhbox\onelinebox
}%
\fi
\else
\hbox to \myboxwidth{%
$\Pi_{21,28}$ spans $L_{251.3}$%
\hfil}%
\hbox to \myboxwidth{%
$322322322262262222622$%
\hfil}%
\box\matricesbox
\fi
}%
\hfill\discretionary{}{}{}%
\setbox\matricesbox=\hbox{%
{$\left[\!\llap{\phantom{%
\begingroup \smaller\smaller\smaller
\endgroup%
}}\!\right]$}%
}%
\ifdim\wd\matricesbox>\halfwidth\myboxwidth=\hsize\else\myboxwidth=\halfwidth\fi
\vbox{%
\ifdim\myboxwidth=\hsize
\setbox\onelinebox=\hbox{%
\vbox{\hbox{%
$\Pi_{21,29}$ spans $L_{251.3}$%
}\hbox{%
$322322322622226222622$%
}%
}%
\hfill\copy\matricesbox
}%
\ifdim\wd\onelinebox>\myboxwidth
\hbox to \myboxwidth{%
$\Pi_{21,29}$ spans $L_{251.3}$%
\hfil
$322322322622226222622$%
}%
\box\matricesbox
\else
\hbox to \myboxwidth{%
\unhbox\onelinebox
}%
\fi
\else
\hbox to \myboxwidth{%
$\Pi_{21,29}$ spans $L_{251.3}$%
\hfil}%
\hbox to \myboxwidth{%
$322322322622226222622$%
\hfil}%
\box\matricesbox
\fi
}%
\hfill\discretionary{}{}{}%
\setbox\matricesbox=\hbox{%
{$\left[\!\llap{\phantom{%
\begingroup \smaller\smaller\smaller
\endgroup%
}}\!\right]$}%
}%
\ifdim\wd\matricesbox>\halfwidth\myboxwidth=\hsize\else\myboxwidth=\halfwidth\fi
\vbox{%
\ifdim\myboxwidth=\hsize
\setbox\onelinebox=\hbox{%
\vbox{\hbox{%
$\Pi_{21,30}$ spans $L_{251.3}$%
}\hbox{%
$322322262232226222622$%
}%
}%
\hfill\copy\matricesbox
}%
\ifdim\wd\onelinebox>\myboxwidth
\hbox to \myboxwidth{%
$\Pi_{21,30}$ spans $L_{251.3}$%
\hfil
$322322262232226222622$%
}%
\box\matricesbox
\else
\hbox to \myboxwidth{%
\unhbox\onelinebox
}%
\fi
\else
\hbox to \myboxwidth{%
$\Pi_{21,30}$ spans $L_{251.3}$%
\hfil}%
\hbox to \myboxwidth{%
$322322262232226222622$%
\hfil}%
\box\matricesbox
\fi
}%
\hfill\discretionary{}{}{}%
\setbox\matricesbox=\hbox{%
{$\left[\!\llap{\phantom{%
\begingroup \smaller\smaller\smaller
\endgroup%
}}\!\right]$}%
}%
\ifdim\wd\matricesbox>\halfwidth\myboxwidth=\hsize\else\myboxwidth=\halfwidth\fi
\vbox{%
\ifdim\myboxwidth=\hsize
\setbox\onelinebox=\hbox{%
\vbox{\hbox{%
$\Pi_{21,31}$ spans $L_{251.3}$%
}\hbox{%
$322322262232226226222$%
}%
}%
\hfill\copy\matricesbox
}%
\ifdim\wd\onelinebox>\myboxwidth
\hbox to \myboxwidth{%
$\Pi_{21,31}$ spans $L_{251.3}$%
\hfil
$322322262232226226222$%
}%
\box\matricesbox
\else
\hbox to \myboxwidth{%
\unhbox\onelinebox
}%
\fi
\else
\hbox to \myboxwidth{%
$\Pi_{21,31}$ spans $L_{251.3}$%
\hfil}%
\hbox to \myboxwidth{%
$322322262232226226222$%
\hfil}%
\box\matricesbox
\fi
}%
\hfill\discretionary{}{}{}%
\setbox\matricesbox=\hbox{%
{$\left[\!\llap{\phantom{%
\begingroup \smaller\smaller\smaller
\endgroup%
}}\!\right]$}%
}%
\ifdim\wd\matricesbox>\halfwidth\myboxwidth=\hsize\else\myboxwidth=\halfwidth\fi
\vbox{%
\ifdim\myboxwidth=\hsize
\setbox\onelinebox=\hbox{%
\vbox{\hbox{%
$\Pi_{21,32}$ spans $L_{251.3}$%
}\hbox{%
$322322262232262222622$%
}%
}%
\hfill\copy\matricesbox
}%
\ifdim\wd\onelinebox>\myboxwidth
\hbox to \myboxwidth{%
$\Pi_{21,32}$ spans $L_{251.3}$%
\hfil
$322322262232262222622$%
}%
\box\matricesbox
\else
\hbox to \myboxwidth{%
\unhbox\onelinebox
}%
\fi
\else
\hbox to \myboxwidth{%
$\Pi_{21,32}$ spans $L_{251.3}$%
\hfil}%
\hbox to \myboxwidth{%
$322322262232262222622$%
\hfil}%
\box\matricesbox
\fi
}%
\hfill\discretionary{}{}{}%
\setbox\matricesbox=\hbox{%
{$\left[\!\llap{\phantom{%
\begingroup \smaller\smaller\smaller
\endgroup%
}}\!\right]$}%
}%
\ifdim\wd\matricesbox>\halfwidth\myboxwidth=\hsize\else\myboxwidth=\halfwidth\fi
\vbox{%
\ifdim\myboxwidth=\hsize
\setbox\onelinebox=\hbox{%
\vbox{\hbox{%
$\Pi_{21,33}$ spans $L_{251.3}$%
}\hbox{%
$322322262232262226222$%
}%
}%
\hfill\copy\matricesbox
}%
\ifdim\wd\onelinebox>\myboxwidth
\hbox to \myboxwidth{%
$\Pi_{21,33}$ spans $L_{251.3}$%
\hfil
$322322262232262226222$%
}%
\box\matricesbox
\else
\hbox to \myboxwidth{%
\unhbox\onelinebox
}%
\fi
\else
\hbox to \myboxwidth{%
$\Pi_{21,33}$ spans $L_{251.3}$%
\hfil}%
\hbox to \myboxwidth{%
$322322262232262226222$%
\hfil}%
\box\matricesbox
\fi
}%
\hfill\discretionary{}{}{}%
\setbox\matricesbox=\hbox{%
{$\left[\!\llap{\phantom{%
\begingroup \smaller\smaller\smaller
\endgroup%
}}\!\right]$}%
}%
\ifdim\wd\matricesbox>\halfwidth\myboxwidth=\hsize\else\myboxwidth=\halfwidth\fi
\vbox{%
\ifdim\myboxwidth=\hsize
\setbox\onelinebox=\hbox{%
\vbox{\hbox{%
$\Pi_{21,34}$ spans $L_{251.3}$%
}\hbox{%
$322322262226223222622$%
}%
}%
\hfill\copy\matricesbox
}%
\ifdim\wd\onelinebox>\myboxwidth
\hbox to \myboxwidth{%
$\Pi_{21,34}$ spans $L_{251.3}$%
\hfil
$322322262226223222622$%
}%
\box\matricesbox
\else
\hbox to \myboxwidth{%
\unhbox\onelinebox
}%
\fi
\else
\hbox to \myboxwidth{%
$\Pi_{21,34}$ spans $L_{251.3}$%
\hfil}%
\hbox to \myboxwidth{%
$322322262226223222622$%
\hfil}%
\box\matricesbox
\fi
}%
\hfill\discretionary{}{}{}%
\setbox\matricesbox=\hbox{%
{$\left[\!\llap{\phantom{%
\begingroup \smaller\smaller\smaller
\endgroup%
}}\!\right]$}%
}%
\ifdim\wd\matricesbox>\halfwidth\myboxwidth=\hsize\else\myboxwidth=\halfwidth\fi
\vbox{%
\ifdim\myboxwidth=\hsize
\setbox\onelinebox=\hbox{%
\vbox{\hbox{%
$\Pi_{21,35}$ spans $L_{251.3}$%
}\hbox{%
$322322262262223222622$%
}%
}%
\hfill\copy\matricesbox
}%
\ifdim\wd\onelinebox>\myboxwidth
\hbox to \myboxwidth{%
$\Pi_{21,35}$ spans $L_{251.3}$%
\hfil
$322322262262223222622$%
}%
\box\matricesbox
\else
\hbox to \myboxwidth{%
\unhbox\onelinebox
}%
\fi
\else
\hbox to \myboxwidth{%
$\Pi_{21,35}$ spans $L_{251.3}$%
\hfil}%
\hbox to \myboxwidth{%
$322322262262223222622$%
\hfil}%
\box\matricesbox
\fi
}%
\hfill\discretionary{}{}{}%
\setbox\matricesbox=\hbox{%
{$\left[\!\llap{\phantom{%
\begingroup \smaller\smaller\smaller
\endgroup%
}}\!\right]$}%
}%
\ifdim\wd\matricesbox>\halfwidth\myboxwidth=\hsize\else\myboxwidth=\halfwidth\fi
\vbox{%
\ifdim\myboxwidth=\hsize
\setbox\onelinebox=\hbox{%
\vbox{\hbox{%
$\Pi_{21,36}$ spans $L_{251.3}$%
}\hbox{%
$322322622232226222622$%
}%
}%
\hfill\copy\matricesbox
}%
\ifdim\wd\onelinebox>\myboxwidth
\hbox to \myboxwidth{%
$\Pi_{21,36}$ spans $L_{251.3}$%
\hfil
$322322622232226222622$%
}%
\box\matricesbox
\else
\hbox to \myboxwidth{%
\unhbox\onelinebox
}%
\fi
\else
\hbox to \myboxwidth{%
$\Pi_{21,36}$ spans $L_{251.3}$%
\hfil}%
\hbox to \myboxwidth{%
$322322622232226222622$%
\hfil}%
\box\matricesbox
\fi
}%
\hfill\discretionary{}{}{}%
\setbox\matricesbox=\hbox{%
{$\left[\!\llap{\phantom{%
\begingroup \smaller\smaller\smaller
\endgroup%
}}\!\right]$}%
}%
\ifdim\wd\matricesbox>\halfwidth\myboxwidth=\hsize\else\myboxwidth=\halfwidth\fi
\vbox{%
\ifdim\myboxwidth=\hsize
\setbox\onelinebox=\hbox{%
\vbox{\hbox{%
$\Pi_{21,37}$ spans $L_{251.3}$%
}\hbox{%
$322322622232262222622$%
}%
}%
\hfill\copy\matricesbox
}%
\ifdim\wd\onelinebox>\myboxwidth
\hbox to \myboxwidth{%
$\Pi_{21,37}$ spans $L_{251.3}$%
\hfil
$322322622232262222622$%
}%
\box\matricesbox
\else
\hbox to \myboxwidth{%
\unhbox\onelinebox
}%
\fi
\else
\hbox to \myboxwidth{%
$\Pi_{21,37}$ spans $L_{251.3}$%
\hfil}%
\hbox to \myboxwidth{%
$322322622232262222622$%
\hfil}%
\box\matricesbox
\fi
}%
\hfill\discretionary{}{}{}%
\setbox\matricesbox=\hbox{%
{$\left[\!\llap{\phantom{%
\begingroup \smaller\smaller\smaller
\endgroup%
}}\!\right]$}%
}%
\ifdim\wd\matricesbox>\halfwidth\myboxwidth=\hsize\else\myboxwidth=\halfwidth\fi
\vbox{%
\ifdim\myboxwidth=\hsize
\setbox\onelinebox=\hbox{%
\vbox{\hbox{%
$\Pi_{21,38}$ spans $L_{251.3}$%
}\hbox{%
$322262232226223226222$%
}%
}%
\hfill\copy\matricesbox
}%
\ifdim\wd\onelinebox>\myboxwidth
\hbox to \myboxwidth{%
$\Pi_{21,38}$ spans $L_{251.3}$%
\hfil
$322262232226223226222$%
}%
\box\matricesbox
\else
\hbox to \myboxwidth{%
\unhbox\onelinebox
}%
\fi
\else
\hbox to \myboxwidth{%
$\Pi_{21,38}$ spans $L_{251.3}$%
\hfil}%
\hbox to \myboxwidth{%
$322262232226223226222$%
\hfil}%
\box\matricesbox
\fi
}%
\hfill\discretionary{}{}{}%

\vskip2pt\hrule\vskip2pt

\leavevmode\setbox\matricesbox=\hbox{%
{$\left[\!\llap{\phantom{%
\begingroup \smaller\smaller\smaller\begin{tabular}{@{}c@{}}%
\phantom{0}\\\phantom{0}\\\phantom{0}
\end{tabular}\endgroup%
}}\right.$}%
\begingroup \smaller\smaller\smaller\begin{tabular}{@{}c@{}}%
-3\\\phantom{0}\\\phantom{0}
\end{tabular}\endgroup%
\kern3pt%
\begingroup \smaller\smaller\smaller\begin{tabular}{@{}c@{}}%
\phantom{0}\\15/2\\\phantom{0}
\end{tabular}\endgroup%
\kern3pt%
\begingroup \smaller\smaller\smaller\begin{tabular}{@{}c@{}}%
\phantom{0}\\\phantom{0}\\5/2
\end{tabular}\endgroup%
{$\left.\llap{\phantom{%
\begingroup \smaller\smaller\smaller\begin{tabular}{@{}c@{}}%
\phantom{0}\\\phantom{0}\\\phantom{0}
\end{tabular}\endgroup%
}}\!\right]$}%
{$\left[\!\llap{\phantom{%
\begingroup \smaller\smaller\smaller\begin{tabular}{@{}c@{}}%
0\\0\\0
\end{tabular}\endgroup%
}}\right.$}%
\begingroup \smaller\smaller\smaller\begin{tabular}{@{}c@{}}%
3\\2\\0
\end{tabular}\endgroup%
\kern3pt%
\begingroup \smaller\smaller\smaller\begin{tabular}{@{}c@{}}%
10\\6\\-4
\end{tabular}\endgroup%
\kern3pt%
\begingroup \smaller\smaller\smaller\begin{tabular}{@{}c@{}}%
10\\5\\-7
\end{tabular}\endgroup%
\kern3pt%
\begingroup \smaller\smaller\smaller\begin{tabular}{@{}c@{}}%
30\\11\\-27
\end{tabular}\endgroup%
\kern3pt%
\begingroup \smaller\smaller\smaller\begin{tabular}{@{}c@{}}%
30\\8\\-30
\end{tabular}\endgroup%
\kern3pt%
\begingroup \smaller\smaller\smaller\begin{tabular}{@{}c@{}}%
10\\1\\-11
\end{tabular}\endgroup%
{$\left.\llap{\phantom{%
\begingroup \smaller\smaller\smaller\begin{tabular}{@{}c@{}}%
0\\0\\0
\end{tabular}\endgroup%
}}\!\right]$}%
}%
\ifdim\wd\matricesbox>\halfwidth\myboxwidth=\hsize\else\myboxwidth=\halfwidth\fi
\vbox{%
\ifdim\myboxwidth=\hsize
\setbox\onelinebox=\hbox{%
\vbox{\hbox{%
$\Pi_{22,1}$ spans $L_{16.9}$%
}\hbox{%
$3632|23636\slashthree63632|23636\slashthree6\rtimes D_{4}$%
}%
}%
\hfill\copy\matricesbox
}%
\ifdim\wd\onelinebox>\myboxwidth
\hbox to \myboxwidth{%
$\Pi_{22,1}$ spans $L_{16.9}$%
\hfil
$3632|23636\slashthree63632|23636\slashthree6\rtimes D_{4}$%
}%
\box\matricesbox
\else
\hbox to \myboxwidth{%
\unhbox\onelinebox
}%
\fi
\else
\hbox to \myboxwidth{%
$\Pi_{22,1}$ spans $L_{16.9}$%
\hfil}%
\hbox to \myboxwidth{%
$3632|23636\slashthree63632|23636\slashthree6\rtimes D_{4}$%
\hfil}%
\box\matricesbox
\fi
}%
\hfill\discretionary{}{}{}%
\setbox\matricesbox=\hbox{%
{$\left[\!\llap{\phantom{%
\begingroup \smaller\smaller\smaller\begin{tabular}{@{}c@{}}%
\phantom{0}\\\phantom{0}\\\phantom{0}
\end{tabular}\endgroup%
}}\right.$}%
\begingroup \smaller\smaller\smaller\begin{tabular}{@{}c@{}}%
-5\\\phantom{0}\\\phantom{0}
\end{tabular}\endgroup%
\kern3pt%
\begingroup \smaller\smaller\smaller\begin{tabular}{@{}c@{}}%
\phantom{0}\\9/2\\\phantom{0}
\end{tabular}\endgroup%
\kern3pt%
\begingroup \smaller\smaller\smaller\begin{tabular}{@{}c@{}}%
\phantom{0}\\\phantom{0}\\3/2
\end{tabular}\endgroup%
{$\left.\llap{\phantom{%
\begingroup \smaller\smaller\smaller\begin{tabular}{@{}c@{}}%
\phantom{0}\\\phantom{0}\\\phantom{0}
\end{tabular}\endgroup%
}}\!\right]$}%
{$\left[\!\llap{\phantom{%
\begingroup \smaller\smaller\smaller\begin{tabular}{@{}c@{}}%
0\\0\\0
\end{tabular}\endgroup%
}}\right.$}%
\begingroup \smaller\smaller\smaller\begin{tabular}{@{}c@{}}%
18\\19\\3
\end{tabular}\endgroup%
\kern3pt%
\begingroup \smaller\smaller\smaller\begin{tabular}{@{}c@{}}%
6\\6\\4
\end{tabular}\endgroup%
\kern3pt%
\begingroup \smaller\smaller\smaller\begin{tabular}{@{}c@{}}%
6\\5\\7
\end{tabular}\endgroup%
\kern3pt%
\begingroup \smaller\smaller\smaller\begin{tabular}{@{}c@{}}%
18\\11\\27
\end{tabular}\endgroup%
\kern3pt%
\begingroup \smaller\smaller\smaller\begin{tabular}{@{}c@{}}%
18\\8\\30
\end{tabular}\endgroup%
\kern3pt%
\begingroup \smaller\smaller\smaller\begin{tabular}{@{}c@{}}%
1\\0\\2
\end{tabular}\endgroup%
{$\left.\llap{\phantom{%
\begingroup \smaller\smaller\smaller\begin{tabular}{@{}c@{}}%
0\\0\\0
\end{tabular}\endgroup%
}}\!\right]$}%
}%
\ifdim\wd\matricesbox>\halfwidth\myboxwidth=\hsize\else\myboxwidth=\halfwidth\fi
\vbox{%
\ifdim\myboxwidth=\hsize
\setbox\onelinebox=\hbox{%
\vbox{\hbox{%
$\Pi_{22,2}$ spans $L_{16.7}$%
}\hbox{%
$\slashthree63632|23636\slashthree63632|23636\rtimes D_{4}$%
}%
}%
\hfill\copy\matricesbox
}%
\ifdim\wd\onelinebox>\myboxwidth
\hbox to \myboxwidth{%
$\Pi_{22,2}$ spans $L_{16.7}$%
\hfil
$\slashthree63632|23636\slashthree63632|23636\rtimes D_{4}$%
}%
\box\matricesbox
\else
\hbox to \myboxwidth{%
\unhbox\onelinebox
}%
\fi
\else
\hbox to \myboxwidth{%
$\Pi_{22,2}$ spans $L_{16.7}$%
\hfil}%
\hbox to \myboxwidth{%
$\slashthree63632|23636\slashthree63632|23636\rtimes D_{4}$%
\hfil}%
\box\matricesbox
\fi
}%
\hfill\discretionary{}{}{}%
\setbox\matricesbox=\hbox{%
{$\left[\!\llap{\phantom{%
\begingroup \smaller\smaller\smaller\begin{tabular}{@{}c@{}}%
\phantom{0}\\\phantom{0}\\\phantom{0}
\end{tabular}\endgroup%
}}\right.$}%
\begingroup \smaller\smaller\smaller\begin{tabular}{@{}c@{}}%
-5\\\phantom{0}\\\phantom{0}
\end{tabular}\endgroup%
\kern3pt%
\begingroup \smaller\smaller\smaller\begin{tabular}{@{}c@{}}%
\phantom{0}\\9\\\phantom{0}
\end{tabular}\endgroup%
\kern3pt%
\begingroup \smaller\smaller\smaller\begin{tabular}{@{}c@{}}%
\phantom{0}\\\phantom{0}\\3
\end{tabular}\endgroup%
{$\left.\llap{\phantom{%
\begingroup \smaller\smaller\smaller\begin{tabular}{@{}c@{}}%
\phantom{0}\\\phantom{0}\\\phantom{0}
\end{tabular}\endgroup%
}}\!\right]$}%
{$\left[\!\llap{\phantom{%
\begingroup \smaller\smaller\smaller\begin{tabular}{@{}c@{}}%
0\\0\\0
\end{tabular}\endgroup%
}}\right.$}%
\begingroup \smaller\smaller\smaller\begin{tabular}{@{}c@{}}%
4\\3\\-1
\end{tabular}\endgroup%
\kern3pt%
\begingroup \smaller\smaller\smaller\begin{tabular}{@{}c@{}}%
3\\2\\-2
\end{tabular}\endgroup%
\kern3pt%
\begingroup \smaller\smaller\smaller\begin{tabular}{@{}c@{}}%
4\\2\\-4
\end{tabular}\endgroup%
\kern3pt%
\begingroup \smaller\smaller\smaller\begin{tabular}{@{}c@{}}%
12\\4\\-14
\end{tabular}\endgroup%
\kern3pt%
\begingroup \smaller\smaller\smaller\begin{tabular}{@{}c@{}}%
36\\7\\-45
\end{tabular}\endgroup%
\kern3pt%
\begingroup \smaller\smaller\smaller\begin{tabular}{@{}c@{}}%
3\\0\\-4
\end{tabular}\endgroup%
{$\left.\llap{\phantom{%
\begingroup \smaller\smaller\smaller\begin{tabular}{@{}c@{}}%
0\\0\\0
\end{tabular}\endgroup%
}}\!\right]$}%
}%
\ifdim\wd\matricesbox>\halfwidth\myboxwidth=\hsize\else\myboxwidth=\halfwidth\fi
\vbox{%
\ifdim\myboxwidth=\hsize
\setbox\onelinebox=\hbox{%
\vbox{\hbox{%
$\Pi_{22,3}$ spans $L_{251.3}$%
}\hbox{%
$\slashthree22262|26222\slashthree22262|26222\rtimes D_{4}$%
}%
}%
\hfill\copy\matricesbox
}%
\ifdim\wd\onelinebox>\myboxwidth
\hbox to \myboxwidth{%
$\Pi_{22,3}$ spans $L_{251.3}$%
\hfil
$\slashthree22262|26222\slashthree22262|26222\rtimes D_{4}$%
}%
\box\matricesbox
\else
\hbox to \myboxwidth{%
\unhbox\onelinebox
}%
\fi
\else
\hbox to \myboxwidth{%
$\Pi_{22,3}$ spans $L_{251.3}$%
\hfil}%
\hbox to \myboxwidth{%
$\slashthree22262|26222\slashthree22262|26222\rtimes D_{4}$%
\hfil}%
\box\matricesbox
\fi
}%
\hfill\discretionary{}{}{}%
\setbox\matricesbox=\hbox{%
{$\left[\!\llap{\phantom{%
\begingroup \smaller\smaller\smaller\begin{tabular}{@{}c@{}}%
\phantom{0}\\\phantom{0}\\\phantom{0}
\end{tabular}\endgroup%
}}\right.$}%
\begingroup \smaller\smaller\smaller\begin{tabular}{@{}c@{}}%
-5\\\phantom{0}\\\phantom{0}
\end{tabular}\endgroup%
\kern3pt%
\begingroup \smaller\smaller\smaller\begin{tabular}{@{}c@{}}%
\phantom{0}\\9\\\phantom{0}
\end{tabular}\endgroup%
\kern3pt%
\begingroup \smaller\smaller\smaller\begin{tabular}{@{}c@{}}%
\phantom{0}\\\phantom{0}\\3
\end{tabular}\endgroup%
{$\left.\llap{\phantom{%
\begingroup \smaller\smaller\smaller\begin{tabular}{@{}c@{}}%
\phantom{0}\\\phantom{0}\\\phantom{0}
\end{tabular}\endgroup%
}}\!\right]$}%
{$\left[\!\llap{\phantom{%
\begingroup \smaller\smaller\smaller\begin{tabular}{@{}c@{}}%
0\\0\\0
\end{tabular}\endgroup%
}}\right.$}%
\begingroup \smaller\smaller\smaller\begin{tabular}{@{}c@{}}%
4\\3\\-1
\end{tabular}\endgroup%
\kern3pt%
\begingroup \smaller\smaller\smaller\begin{tabular}{@{}c@{}}%
3\\2\\-2
\end{tabular}\endgroup%
\kern3pt%
\begingroup \smaller\smaller\smaller\begin{tabular}{@{}c@{}}%
36\\19\\-33
\end{tabular}\endgroup%
\kern3pt%
\begingroup \smaller\smaller\smaller\begin{tabular}{@{}c@{}}%
12\\5\\-13
\end{tabular}\endgroup%
\kern3pt%
\begingroup \smaller\smaller\smaller\begin{tabular}{@{}c@{}}%
4\\1\\-5
\end{tabular}\endgroup%
\kern3pt%
\begingroup \smaller\smaller\smaller\begin{tabular}{@{}c@{}}%
3\\0\\-4
\end{tabular}\endgroup%
{$\left.\llap{\phantom{%
\begingroup \smaller\smaller\smaller\begin{tabular}{@{}c@{}}%
0\\0\\0
\end{tabular}\endgroup%
}}\!\right]$}%
}%
\ifdim\wd\matricesbox>\halfwidth\myboxwidth=\hsize\else\myboxwidth=\halfwidth\fi
\vbox{%
\ifdim\myboxwidth=\hsize
\setbox\onelinebox=\hbox{%
\vbox{\hbox{%
$\Pi_{22,4}$ spans $L_{251.3}$%
}\hbox{%
$\slashthree22622|22622\slashthree22622|22622\rtimes D_{4}$%
}%
}%
\hfill\copy\matricesbox
}%
\ifdim\wd\onelinebox>\myboxwidth
\hbox to \myboxwidth{%
$\Pi_{22,4}$ spans $L_{251.3}$%
\hfil
$\slashthree22622|22622\slashthree22622|22622\rtimes D_{4}$%
}%
\box\matricesbox
\else
\hbox to \myboxwidth{%
\unhbox\onelinebox
}%
\fi
\else
\hbox to \myboxwidth{%
$\Pi_{22,4}$ spans $L_{251.3}$%
\hfil}%
\hbox to \myboxwidth{%
$\slashthree22622|22622\slashthree22622|22622\rtimes D_{4}$%
\hfil}%
\box\matricesbox
\fi
}%
\hfill\discretionary{}{}{}%
\setbox\matricesbox=\hbox{%
{$\left[\!\llap{\phantom{%
\begingroup \smaller\smaller\smaller
\endgroup%
}}\!\right]$}%
}%
\ifdim\wd\matricesbox>\halfwidth\myboxwidth=\hsize\else\myboxwidth=\halfwidth\fi
\vbox{%
\ifdim\myboxwidth=\hsize
\setbox\onelinebox=\hbox{%
\vbox{\hbox{%
$\Pi_{22,5}$ spans $L_{16.9}$%
}\hbox{%
$36322\slashthree2236363636\slashthree63636\rtimes D_{2}$%
}%
}%
\hfill\copy\matricesbox
}%
\ifdim\wd\onelinebox>\myboxwidth
\hbox to \myboxwidth{%
$\Pi_{22,5}$ spans $L_{16.9}$%
\hfil
$36322\slashthree2236363636\slashthree63636\rtimes D_{2}$%
}%
\box\matricesbox
\else
\hbox to \myboxwidth{%
\unhbox\onelinebox
}%
\fi
\else
\hbox to \myboxwidth{%
$\Pi_{22,5}$ spans $L_{16.9}$%
\hfil}%
\hbox to \myboxwidth{%
$36322\slashthree2236363636\slashthree63636\rtimes D_{2}$%
\hfil}%
\box\matricesbox
\fi
}%
\hfill\discretionary{}{}{}%
\setbox\matricesbox=\hbox{%
{$\left[\!\llap{\phantom{%
\begingroup \smaller\smaller\smaller
\endgroup%
}}\!\right]$}%
}%
\ifdim\wd\matricesbox>\halfwidth\myboxwidth=\hsize\else\myboxwidth=\halfwidth\fi
\vbox{%
\ifdim\myboxwidth=\hsize
\setbox\onelinebox=\hbox{%
\vbox{\hbox{%
$\Pi_{22,6}$ spans $L_{16.9}$%
}\hbox{%
$3632236\slashthree6322363636\slashthree636\rtimes D_{2}$%
}%
}%
\hfill\copy\matricesbox
}%
\ifdim\wd\onelinebox>\myboxwidth
\hbox to \myboxwidth{%
$\Pi_{22,6}$ spans $L_{16.9}$%
\hfil
$3632236\slashthree6322363636\slashthree636\rtimes D_{2}$%
}%
\box\matricesbox
\else
\hbox to \myboxwidth{%
\unhbox\onelinebox
}%
\fi
\else
\hbox to \myboxwidth{%
$\Pi_{22,6}$ spans $L_{16.9}$%
\hfil}%
\hbox to \myboxwidth{%
$3632236\slashthree6322363636\slashthree636\rtimes D_{2}$%
\hfil}%
\box\matricesbox
\fi
}%
\hfill\discretionary{}{}{}%
\setbox\matricesbox=\hbox{%
{$\left[\!\llap{\phantom{%
\begingroup \smaller\smaller\smaller
\endgroup%
}}\!\right]$}%
}%
\ifdim\wd\matricesbox>\halfwidth\myboxwidth=\hsize\else\myboxwidth=\halfwidth\fi
\vbox{%
\ifdim\myboxwidth=\hsize
\setbox\onelinebox=\hbox{%
\vbox{\hbox{%
$\Pi_{22,7}$ spans $L_{16.7}$%
}\hbox{%
$3636322\slashthree2236363636\slashthree636\rtimes D_{2}$%
}%
}%
\hfill\copy\matricesbox
}%
\ifdim\wd\onelinebox>\myboxwidth
\hbox to \myboxwidth{%
$\Pi_{22,7}$ spans $L_{16.7}$%
\hfil
$3636322\slashthree2236363636\slashthree636\rtimes D_{2}$%
}%
\box\matricesbox
\else
\hbox to \myboxwidth{%
\unhbox\onelinebox
}%
\fi
\else
\hbox to \myboxwidth{%
$\Pi_{22,7}$ spans $L_{16.7}$%
\hfil}%
\hbox to \myboxwidth{%
$3636322\slashthree2236363636\slashthree636\rtimes D_{2}$%
\hfil}%
\box\matricesbox
\fi
}%
\hfill\discretionary{}{}{}%
\setbox\matricesbox=\hbox{%
{$\left[\!\llap{\phantom{%
\begingroup \smaller\smaller\smaller
\endgroup%
}}\!\right]$}%
}%
\ifdim\wd\matricesbox>\halfwidth\myboxwidth=\hsize\else\myboxwidth=\halfwidth\fi
\vbox{%
\ifdim\myboxwidth=\hsize
\setbox\onelinebox=\hbox{%
\vbox{\hbox{%
$\Pi_{22,8}$ spans $L_{16.7}$%
}\hbox{%
$363632236\slashthree6322363636\slashthree6\rtimes D_{2}$%
}%
}%
\hfill\copy\matricesbox
}%
\ifdim\wd\onelinebox>\myboxwidth
\hbox to \myboxwidth{%
$\Pi_{22,8}$ spans $L_{16.7}$%
\hfil
$363632236\slashthree6322363636\slashthree6\rtimes D_{2}$%
}%
\box\matricesbox
\else
\hbox to \myboxwidth{%
\unhbox\onelinebox
}%
\fi
\else
\hbox to \myboxwidth{%
$\Pi_{22,8}$ spans $L_{16.7}$%
\hfil}%
\hbox to \myboxwidth{%
$363632236\slashthree6322363636\slashthree6\rtimes D_{2}$%
\hfil}%
\box\matricesbox
\fi
}%
\hfill\discretionary{}{}{}%
\setbox\matricesbox=\hbox{%
{$\left[\!\llap{\phantom{%
\begingroup \smaller\smaller\smaller
\endgroup%
}}\!\right]$}%
}%
\ifdim\wd\matricesbox>\halfwidth\myboxwidth=\hsize\else\myboxwidth=\halfwidth\fi
\vbox{%
\ifdim\myboxwidth=\hsize
\setbox\onelinebox=\hbox{%
\vbox{\hbox{%
$\Pi_{22,9}$ spans $L_{251.3}$%
}\hbox{%
$32|23222622262|262226222\rtimes D_{2}$%
}%
}%
\hfill\copy\matricesbox
}%
\ifdim\wd\onelinebox>\myboxwidth
\hbox to \myboxwidth{%
$\Pi_{22,9}$ spans $L_{251.3}$%
\hfil
$32|23222622262|262226222\rtimes D_{2}$%
}%
\box\matricesbox
\else
\hbox to \myboxwidth{%
\unhbox\onelinebox
}%
\fi
\else
\hbox to \myboxwidth{%
$\Pi_{22,9}$ spans $L_{251.3}$%
\hfil}%
\hbox to \myboxwidth{%
$32|23222622262|262226222\rtimes D_{2}$%
\hfil}%
\box\matricesbox
\fi
}%
\hfill\discretionary{}{}{}%
\setbox\matricesbox=\hbox{%
{$\left[\!\llap{\phantom{%
\begingroup \smaller\smaller\smaller
\endgroup%
}}\!\right]$}%
}%
\ifdim\wd\matricesbox>\halfwidth\myboxwidth=\hsize\else\myboxwidth=\halfwidth\fi
\vbox{%
\ifdim\myboxwidth=\hsize
\setbox\onelinebox=\hbox{%
\vbox{\hbox{%
$\Pi_{22,10}$ spans $L_{251.3}$%
}\hbox{%
$32|23222622622|226226222\rtimes D_{2}$%
}%
}%
\hfill\copy\matricesbox
}%
\ifdim\wd\onelinebox>\myboxwidth
\hbox to \myboxwidth{%
$\Pi_{22,10}$ spans $L_{251.3}$%
\hfil
$32|23222622622|226226222\rtimes D_{2}$%
}%
\box\matricesbox
\else
\hbox to \myboxwidth{%
\unhbox\onelinebox
}%
\fi
\else
\hbox to \myboxwidth{%
$\Pi_{22,10}$ spans $L_{251.3}$%
\hfil}%
\hbox to \myboxwidth{%
$32|23222622622|226226222\rtimes D_{2}$%
\hfil}%
\box\matricesbox
\fi
}%
\hfill\discretionary{}{}{}%
\setbox\matricesbox=\hbox{%
{$\left[\!\llap{\phantom{%
\begingroup \smaller\smaller\smaller
\endgroup%
}}\!\right]$}%
}%
\ifdim\wd\matricesbox>\halfwidth\myboxwidth=\hsize\else\myboxwidth=\halfwidth\fi
\vbox{%
\ifdim\myboxwidth=\hsize
\setbox\onelinebox=\hbox{%
\vbox{\hbox{%
$\Pi_{22,11}$ spans $L_{251.3}$%
}\hbox{%
$32|23226222262|262222622\rtimes D_{2}$%
}%
}%
\hfill\copy\matricesbox
}%
\ifdim\wd\onelinebox>\myboxwidth
\hbox to \myboxwidth{%
$\Pi_{22,11}$ spans $L_{251.3}$%
\hfil
$32|23226222262|262222622\rtimes D_{2}$%
}%
\box\matricesbox
\else
\hbox to \myboxwidth{%
\unhbox\onelinebox
}%
\fi
\else
\hbox to \myboxwidth{%
$\Pi_{22,11}$ spans $L_{251.3}$%
\hfil}%
\hbox to \myboxwidth{%
$32|23226222262|262222622\rtimes D_{2}$%
\hfil}%
\box\matricesbox
\fi
}%
\hfill\discretionary{}{}{}%
\setbox\matricesbox=\hbox{%
{$\left[\!\llap{\phantom{%
\begingroup \smaller\smaller\smaller
\endgroup%
}}\!\right]$}%
}%
\ifdim\wd\matricesbox>\halfwidth\myboxwidth=\hsize\else\myboxwidth=\halfwidth\fi
\vbox{%
\ifdim\myboxwidth=\hsize
\setbox\onelinebox=\hbox{%
\vbox{\hbox{%
$\Pi_{22,12}$ spans $L_{251.3}$%
}\hbox{%
$32|23226222622|226222622\rtimes D_{2}$%
}%
}%
\hfill\copy\matricesbox
}%
\ifdim\wd\onelinebox>\myboxwidth
\hbox to \myboxwidth{%
$\Pi_{22,12}$ spans $L_{251.3}$%
\hfil
$32|23226222622|226222622\rtimes D_{2}$%
}%
\box\matricesbox
\else
\hbox to \myboxwidth{%
\unhbox\onelinebox
}%
\fi
\else
\hbox to \myboxwidth{%
$\Pi_{22,12}$ spans $L_{251.3}$%
\hfil}%
\hbox to \myboxwidth{%
$32|23226222622|226222622\rtimes D_{2}$%
\hfil}%
\box\matricesbox
\fi
}%
\hfill\discretionary{}{}{}%
\setbox\matricesbox=\hbox{%
{$\left[\!\llap{\phantom{%
\begingroup \smaller\smaller\smaller
\endgroup%
}}\!\right]$}%
}%
\ifdim\wd\matricesbox>\halfwidth\myboxwidth=\hsize\else\myboxwidth=\halfwidth\fi
\vbox{%
\ifdim\myboxwidth=\hsize
\setbox\onelinebox=\hbox{%
\vbox{\hbox{%
$\Pi_{22,13}$ spans $L_{251.3}$%
}\hbox{%
$\slashthree2226222622\slashthree2262226222\rtimes D_{2}$%
}%
}%
\hfill\copy\matricesbox
}%
\ifdim\wd\onelinebox>\myboxwidth
\hbox to \myboxwidth{%
$\Pi_{22,13}$ spans $L_{251.3}$%
\hfil
$\slashthree2226222622\slashthree2262226222\rtimes D_{2}$%
}%
\box\matricesbox
\else
\hbox to \myboxwidth{%
\unhbox\onelinebox
}%
\fi
\else
\hbox to \myboxwidth{%
$\Pi_{22,13}$ spans $L_{251.3}$%
\hfil}%
\hbox to \myboxwidth{%
$\slashthree2226222622\slashthree2262226222\rtimes D_{2}$%
\hfil}%
\box\matricesbox
\fi
}%
\hfill\discretionary{}{}{}%
\setbox\matricesbox=\hbox{%
{$\left[\!\llap{\phantom{%
\begingroup \smaller\smaller\smaller
\endgroup%
}}\!\right]$}%
}%
\ifdim\wd\matricesbox>\halfwidth\myboxwidth=\hsize\else\myboxwidth=\halfwidth\fi
\vbox{%
\ifdim\myboxwidth=\hsize
\setbox\onelinebox=\hbox{%
\vbox{\hbox{%
$\Pi_{22,14}$ spans $L_{251.3}$%
}\hbox{%
$322262|26222322622|22622\rtimes D_{2}$%
}%
}%
\hfill\copy\matricesbox
}%
\ifdim\wd\onelinebox>\myboxwidth
\hbox to \myboxwidth{%
$\Pi_{22,14}$ spans $L_{251.3}$%
\hfil
$322262|26222322622|22622\rtimes D_{2}$%
}%
\box\matricesbox
\else
\hbox to \myboxwidth{%
\unhbox\onelinebox
}%
\fi
\else
\hbox to \myboxwidth{%
$\Pi_{22,14}$ spans $L_{251.3}$%
\hfil}%
\hbox to \myboxwidth{%
$322262|26222322622|22622\rtimes D_{2}$%
\hfil}%
\box\matricesbox
\fi
}%
\hfill\discretionary{}{}{}%
\setbox\matricesbox=\hbox{%
{$\left[\!\llap{\phantom{%
\begingroup \smaller\smaller\smaller
\endgroup%
}}\!\right]$}%
}%
\ifdim\wd\matricesbox>\halfwidth\myboxwidth=\hsize\else\myboxwidth=\halfwidth\fi
\vbox{%
\ifdim\myboxwidth=\hsize
\setbox\onelinebox=\hbox{%
\vbox{\hbox{%
$\Pi_{22,15}$ spans $L_{251.3}$%
}\hbox{%
$3222622262232226222622\rtimes C_{2}$%
}%
}%
\hfill\copy\matricesbox
}%
\ifdim\wd\onelinebox>\myboxwidth
\hbox to \myboxwidth{%
$\Pi_{22,15}$ spans $L_{251.3}$%
\hfil
$3222622262232226222622\rtimes C_{2}$%
}%
\box\matricesbox
\else
\hbox to \myboxwidth{%
\unhbox\onelinebox
}%
\fi
\else
\hbox to \myboxwidth{%
$\Pi_{22,15}$ spans $L_{251.3}$%
\hfil}%
\hbox to \myboxwidth{%
$3222622262232226222622\rtimes C_{2}$%
\hfil}%
\box\matricesbox
\fi
}%
\hfill\discretionary{}{}{}%
\setbox\matricesbox=\hbox{%
{$\left[\!\llap{\phantom{%
\begingroup \smaller\smaller\smaller
\endgroup%
}}\!\right]$}%
}%
\ifdim\wd\matricesbox>\halfwidth\myboxwidth=\hsize\else\myboxwidth=\halfwidth\fi
\vbox{%
\ifdim\myboxwidth=\hsize
\setbox\onelinebox=\hbox{%
\vbox{\hbox{%
$\Pi_{22,16}$ spans $L_{16.13}$%
}\hbox{%
$3632223636363636363636$%
}%
}%
\hfill\copy\matricesbox
}%
\ifdim\wd\onelinebox>\myboxwidth
\hbox to \myboxwidth{%
$\Pi_{22,16}$ spans $L_{16.13}$%
\hfil
$3632223636363636363636$%
}%
\box\matricesbox
\else
\hbox to \myboxwidth{%
\unhbox\onelinebox
}%
\fi
\else
\hbox to \myboxwidth{%
$\Pi_{22,16}$ spans $L_{16.13}$%
\hfil}%
\hbox to \myboxwidth{%
$3632223636363636363636$%
\hfil}%
\box\matricesbox
\fi
}%
\hfill\discretionary{}{}{}%
\setbox\matricesbox=\hbox{%
{$\left[\!\llap{\phantom{%
\begingroup \smaller\smaller\smaller
\endgroup%
}}\!\right]$}%
}%
\ifdim\wd\matricesbox>\halfwidth\myboxwidth=\hsize\else\myboxwidth=\halfwidth\fi
\vbox{%
\ifdim\myboxwidth=\hsize
\setbox\onelinebox=\hbox{%
\vbox{\hbox{%
$\Pi_{22,17}$ spans $L_{16.13}$%
}\hbox{%
$3632236322363636363636$%
}%
}%
\hfill\copy\matricesbox
}%
\ifdim\wd\onelinebox>\myboxwidth
\hbox to \myboxwidth{%
$\Pi_{22,17}$ spans $L_{16.13}$%
\hfil
$3632236322363636363636$%
}%
\box\matricesbox
\else
\hbox to \myboxwidth{%
\unhbox\onelinebox
}%
\fi
\else
\hbox to \myboxwidth{%
$\Pi_{22,17}$ spans $L_{16.13}$%
\hfil}%
\hbox to \myboxwidth{%
$3632236322363636363636$%
\hfil}%
\box\matricesbox
\fi
}%
\hfill\discretionary{}{}{}%
\setbox\matricesbox=\hbox{%
{$\left[\!\llap{\phantom{%
\begingroup \smaller\smaller\smaller
\endgroup%
}}\!\right]$}%
}%
\ifdim\wd\matricesbox>\halfwidth\myboxwidth=\hsize\else\myboxwidth=\halfwidth\fi
\vbox{%
\ifdim\myboxwidth=\hsize
\setbox\onelinebox=\hbox{%
\vbox{\hbox{%
$\Pi_{22,18}$ spans $L_{16.13}$%
}\hbox{%
$3632236363632236363636$%
}%
}%
\hfill\copy\matricesbox
}%
\ifdim\wd\onelinebox>\myboxwidth
\hbox to \myboxwidth{%
$\Pi_{22,18}$ spans $L_{16.13}$%
\hfil
$3632236363632236363636$%
}%
\box\matricesbox
\else
\hbox to \myboxwidth{%
\unhbox\onelinebox
}%
\fi
\else
\hbox to \myboxwidth{%
$\Pi_{22,18}$ spans $L_{16.13}$%
\hfil}%
\hbox to \myboxwidth{%
$3632236363632236363636$%
\hfil}%
\box\matricesbox
\fi
}%
\hfill\discretionary{}{}{}%
\setbox\matricesbox=\hbox{%
{$\left[\!\llap{\phantom{%
\begingroup \smaller\smaller\smaller
\endgroup%
}}\!\right]$}%
}%
\ifdim\wd\matricesbox>\halfwidth\myboxwidth=\hsize\else\myboxwidth=\halfwidth\fi
\vbox{%
\ifdim\myboxwidth=\hsize
\setbox\onelinebox=\hbox{%
\vbox{\hbox{%
$\Pi_{22,19}$ spans $L_{251.3}$%
}\hbox{%
$3223222622262226222622$%
}%
}%
\hfill\copy\matricesbox
}%
\ifdim\wd\onelinebox>\myboxwidth
\hbox to \myboxwidth{%
$\Pi_{22,19}$ spans $L_{251.3}$%
\hfil
$3223222622262226222622$%
}%
\box\matricesbox
\else
\hbox to \myboxwidth{%
\unhbox\onelinebox
}%
\fi
\else
\hbox to \myboxwidth{%
$\Pi_{22,19}$ spans $L_{251.3}$%
\hfil}%
\hbox to \myboxwidth{%
$3223222622262226222622$%
\hfil}%
\box\matricesbox
\fi
}%
\hfill\discretionary{}{}{}%
\setbox\matricesbox=\hbox{%
{$\left[\!\llap{\phantom{%
\begingroup \smaller\smaller\smaller
\endgroup%
}}\!\right]$}%
}%
\ifdim\wd\matricesbox>\halfwidth\myboxwidth=\hsize\else\myboxwidth=\halfwidth\fi
\vbox{%
\ifdim\myboxwidth=\hsize
\setbox\onelinebox=\hbox{%
\vbox{\hbox{%
$\Pi_{22,20}$ spans $L_{251.3}$%
}\hbox{%
$3223222622262226226222$%
}%
}%
\hfill\copy\matricesbox
}%
\ifdim\wd\onelinebox>\myboxwidth
\hbox to \myboxwidth{%
$\Pi_{22,20}$ spans $L_{251.3}$%
\hfil
$3223222622262226226222$%
}%
\box\matricesbox
\else
\hbox to \myboxwidth{%
\unhbox\onelinebox
}%
\fi
\else
\hbox to \myboxwidth{%
$\Pi_{22,20}$ spans $L_{251.3}$%
\hfil}%
\hbox to \myboxwidth{%
$3223222622262226226222$%
\hfil}%
\box\matricesbox
\fi
}%
\hfill\discretionary{}{}{}%
\setbox\matricesbox=\hbox{%
{$\left[\!\llap{\phantom{%
\begingroup \smaller\smaller\smaller
\endgroup%
}}\!\right]$}%
}%
\ifdim\wd\matricesbox>\halfwidth\myboxwidth=\hsize\else\myboxwidth=\halfwidth\fi
\vbox{%
\ifdim\myboxwidth=\hsize
\setbox\onelinebox=\hbox{%
\vbox{\hbox{%
$\Pi_{22,21}$ spans $L_{251.3}$%
}\hbox{%
$3223222622262262222622$%
}%
}%
\hfill\copy\matricesbox
}%
\ifdim\wd\onelinebox>\myboxwidth
\hbox to \myboxwidth{%
$\Pi_{22,21}$ spans $L_{251.3}$%
\hfil
$3223222622262262222622$%
}%
\box\matricesbox
\else
\hbox to \myboxwidth{%
\unhbox\onelinebox
}%
\fi
\else
\hbox to \myboxwidth{%
$\Pi_{22,21}$ spans $L_{251.3}$%
\hfil}%
\hbox to \myboxwidth{%
$3223222622262262222622$%
\hfil}%
\box\matricesbox
\fi
}%
\hfill\discretionary{}{}{}%
\setbox\matricesbox=\hbox{%
{$\left[\!\llap{\phantom{%
\begingroup \smaller\smaller\smaller
\endgroup%
}}\!\right]$}%
}%
\ifdim\wd\matricesbox>\halfwidth\myboxwidth=\hsize\else\myboxwidth=\halfwidth\fi
\vbox{%
\ifdim\myboxwidth=\hsize
\setbox\onelinebox=\hbox{%
\vbox{\hbox{%
$\Pi_{22,22}$ spans $L_{251.3}$%
}\hbox{%
$3223222622622226222622$%
}%
}%
\hfill\copy\matricesbox
}%
\ifdim\wd\onelinebox>\myboxwidth
\hbox to \myboxwidth{%
$\Pi_{22,22}$ spans $L_{251.3}$%
\hfil
$3223222622622226222622$%
}%
\box\matricesbox
\else
\hbox to \myboxwidth{%
\unhbox\onelinebox
}%
\fi
\else
\hbox to \myboxwidth{%
$\Pi_{22,22}$ spans $L_{251.3}$%
\hfil}%
\hbox to \myboxwidth{%
$3223222622622226222622$%
\hfil}%
\box\matricesbox
\fi
}%
\hfill\discretionary{}{}{}%
\setbox\matricesbox=\hbox{%
{$\left[\!\llap{\phantom{%
\begingroup \smaller\smaller\smaller
\endgroup%
}}\!\right]$}%
}%
\ifdim\wd\matricesbox>\halfwidth\myboxwidth=\hsize\else\myboxwidth=\halfwidth\fi
\vbox{%
\ifdim\myboxwidth=\hsize
\setbox\onelinebox=\hbox{%
\vbox{\hbox{%
$\Pi_{22,23}$ spans $L_{251.3}$%
}\hbox{%
$3223222622622262222622$%
}%
}%
\hfill\copy\matricesbox
}%
\ifdim\wd\onelinebox>\myboxwidth
\hbox to \myboxwidth{%
$\Pi_{22,23}$ spans $L_{251.3}$%
\hfil
$3223222622622262222622$%
}%
\box\matricesbox
\else
\hbox to \myboxwidth{%
\unhbox\onelinebox
}%
\fi
\else
\hbox to \myboxwidth{%
$\Pi_{22,23}$ spans $L_{251.3}$%
\hfil}%
\hbox to \myboxwidth{%
$3223222622622262222622$%
\hfil}%
\box\matricesbox
\fi
}%
\hfill\discretionary{}{}{}%
\setbox\matricesbox=\hbox{%
{$\left[\!\llap{\phantom{%
\begingroup \smaller\smaller\smaller
\endgroup%
}}\!\right]$}%
}%
\ifdim\wd\matricesbox>\halfwidth\myboxwidth=\hsize\else\myboxwidth=\halfwidth\fi
\vbox{%
\ifdim\myboxwidth=\hsize
\setbox\onelinebox=\hbox{%
\vbox{\hbox{%
$\Pi_{22,24}$ spans $L_{251.3}$%
}\hbox{%
$3223226222262226222622$%
}%
}%
\hfill\copy\matricesbox
}%
\ifdim\wd\onelinebox>\myboxwidth
\hbox to \myboxwidth{%
$\Pi_{22,24}$ spans $L_{251.3}$%
\hfil
$3223226222262226222622$%
}%
\box\matricesbox
\else
\hbox to \myboxwidth{%
\unhbox\onelinebox
}%
\fi
\else
\hbox to \myboxwidth{%
$\Pi_{22,24}$ spans $L_{251.3}$%
\hfil}%
\hbox to \myboxwidth{%
$3223226222262226222622$%
\hfil}%
\box\matricesbox
\fi
}%
\hfill\discretionary{}{}{}%
\setbox\matricesbox=\hbox{%
{$\left[\!\llap{\phantom{%
\begingroup \smaller\smaller\smaller
\endgroup%
}}\!\right]$}%
}%
\ifdim\wd\matricesbox>\halfwidth\myboxwidth=\hsize\else\myboxwidth=\halfwidth\fi
\vbox{%
\ifdim\myboxwidth=\hsize
\setbox\onelinebox=\hbox{%
\vbox{\hbox{%
$\Pi_{22,25}$ spans $L_{251.3}$%
}\hbox{%
$3222622322262226222622$%
}%
}%
\hfill\copy\matricesbox
}%
\ifdim\wd\onelinebox>\myboxwidth
\hbox to \myboxwidth{%
$\Pi_{22,25}$ spans $L_{251.3}$%
\hfil
$3222622322262226222622$%
}%
\box\matricesbox
\else
\hbox to \myboxwidth{%
\unhbox\onelinebox
}%
\fi
\else
\hbox to \myboxwidth{%
$\Pi_{22,25}$ spans $L_{251.3}$%
\hfil}%
\hbox to \myboxwidth{%
$3222622322262226222622$%
\hfil}%
\box\matricesbox
\fi
}%
\hfill\discretionary{}{}{}%
\setbox\matricesbox=\hbox{%
{$\left[\!\llap{\phantom{%
\begingroup \smaller\smaller\smaller
\endgroup%
}}\!\right]$}%
}%
\ifdim\wd\matricesbox>\halfwidth\myboxwidth=\hsize\else\myboxwidth=\halfwidth\fi
\vbox{%
\ifdim\myboxwidth=\hsize
\setbox\onelinebox=\hbox{%
\vbox{\hbox{%
$\Pi_{22,26}$ spans $L_{251.3}$%
}\hbox{%
$3222622322262226226222$%
}%
}%
\hfill\copy\matricesbox
}%
\ifdim\wd\onelinebox>\myboxwidth
\hbox to \myboxwidth{%
$\Pi_{22,26}$ spans $L_{251.3}$%
\hfil
$3222622322262226226222$%
}%
\box\matricesbox
\else
\hbox to \myboxwidth{%
\unhbox\onelinebox
}%
\fi
\else
\hbox to \myboxwidth{%
$\Pi_{22,26}$ spans $L_{251.3}$%
\hfil}%
\hbox to \myboxwidth{%
$3222622322262226226222$%
\hfil}%
\box\matricesbox
\fi
}%
\hfill\discretionary{}{}{}%
\setbox\matricesbox=\hbox{%
{$\left[\!\llap{\phantom{%
\begingroup \smaller\smaller\smaller
\endgroup%
}}\!\right]$}%
}%
\ifdim\wd\matricesbox>\halfwidth\myboxwidth=\hsize\else\myboxwidth=\halfwidth\fi
\vbox{%
\ifdim\myboxwidth=\hsize
\setbox\onelinebox=\hbox{%
\vbox{\hbox{%
$\Pi_{22,27}$ spans $L_{251.3}$%
}\hbox{%
$3222622322262262222622$%
}%
}%
\hfill\copy\matricesbox
}%
\ifdim\wd\onelinebox>\myboxwidth
\hbox to \myboxwidth{%
$\Pi_{22,27}$ spans $L_{251.3}$%
\hfil
$3222622322262262222622$%
}%
\box\matricesbox
\else
\hbox to \myboxwidth{%
\unhbox\onelinebox
}%
\fi
\else
\hbox to \myboxwidth{%
$\Pi_{22,27}$ spans $L_{251.3}$%
\hfil}%
\hbox to \myboxwidth{%
$3222622322262262222622$%
\hfil}%
\box\matricesbox
\fi
}%
\hfill\discretionary{}{}{}%
\setbox\matricesbox=\hbox{%
{$\left[\!\llap{\phantom{%
\begingroup \smaller\smaller\smaller
\endgroup%
}}\!\right]$}%
}%
\ifdim\wd\matricesbox>\halfwidth\myboxwidth=\hsize\else\myboxwidth=\halfwidth\fi
\vbox{%
\ifdim\myboxwidth=\hsize
\setbox\onelinebox=\hbox{%
\vbox{\hbox{%
$\Pi_{22,28}$ spans $L_{251.3}$%
}\hbox{%
$3222622322262262226222$%
}%
}%
\hfill\copy\matricesbox
}%
\ifdim\wd\onelinebox>\myboxwidth
\hbox to \myboxwidth{%
$\Pi_{22,28}$ spans $L_{251.3}$%
\hfil
$3222622322262262226222$%
}%
\box\matricesbox
\else
\hbox to \myboxwidth{%
\unhbox\onelinebox
}%
\fi
\else
\hbox to \myboxwidth{%
$\Pi_{22,28}$ spans $L_{251.3}$%
\hfil}%
\hbox to \myboxwidth{%
$3222622322262262226222$%
\hfil}%
\box\matricesbox
\fi
}%
\hfill\discretionary{}{}{}%
\setbox\matricesbox=\hbox{%
{$\left[\!\llap{\phantom{%
\begingroup \smaller\smaller\smaller
\endgroup%
}}\!\right]$}%
}%
\ifdim\wd\matricesbox>\halfwidth\myboxwidth=\hsize\else\myboxwidth=\halfwidth\fi
\vbox{%
\ifdim\myboxwidth=\hsize
\setbox\onelinebox=\hbox{%
\vbox{\hbox{%
$\Pi_{22,29}$ spans $L_{251.3}$%
}\hbox{%
$3222622322622226222622$%
}%
}%
\hfill\copy\matricesbox
}%
\ifdim\wd\onelinebox>\myboxwidth
\hbox to \myboxwidth{%
$\Pi_{22,29}$ spans $L_{251.3}$%
\hfil
$3222622322622226222622$%
}%
\box\matricesbox
\else
\hbox to \myboxwidth{%
\unhbox\onelinebox
}%
\fi
\else
\hbox to \myboxwidth{%
$\Pi_{22,29}$ spans $L_{251.3}$%
\hfil}%
\hbox to \myboxwidth{%
$3222622322622226222622$%
\hfil}%
\box\matricesbox
\fi
}%
\hfill\discretionary{}{}{}%
\setbox\matricesbox=\hbox{%
{$\left[\!\llap{\phantom{%
\begingroup \smaller\smaller\smaller
\endgroup%
}}\!\right]$}%
}%
\ifdim\wd\matricesbox>\halfwidth\myboxwidth=\hsize\else\myboxwidth=\halfwidth\fi
\vbox{%
\ifdim\myboxwidth=\hsize
\setbox\onelinebox=\hbox{%
\vbox{\hbox{%
$\Pi_{22,30}$ spans $L_{251.3}$%
}\hbox{%
$3222622322622226226222$%
}%
}%
\hfill\copy\matricesbox
}%
\ifdim\wd\onelinebox>\myboxwidth
\hbox to \myboxwidth{%
$\Pi_{22,30}$ spans $L_{251.3}$%
\hfil
$3222622322622226226222$%
}%
\box\matricesbox
\else
\hbox to \myboxwidth{%
\unhbox\onelinebox
}%
\fi
\else
\hbox to \myboxwidth{%
$\Pi_{22,30}$ spans $L_{251.3}$%
\hfil}%
\hbox to \myboxwidth{%
$3222622322622226226222$%
\hfil}%
\box\matricesbox
\fi
}%
\hfill\discretionary{}{}{}%
\setbox\matricesbox=\hbox{%
{$\left[\!\llap{\phantom{%
\begingroup \smaller\smaller\smaller
\endgroup%
}}\!\right]$}%
}%
\ifdim\wd\matricesbox>\halfwidth\myboxwidth=\hsize\else\myboxwidth=\halfwidth\fi
\vbox{%
\ifdim\myboxwidth=\hsize
\setbox\onelinebox=\hbox{%
\vbox{\hbox{%
$\Pi_{22,31}$ spans $L_{251.3}$%
}\hbox{%
$3222622322622262222622$%
}%
}%
\hfill\copy\matricesbox
}%
\ifdim\wd\onelinebox>\myboxwidth
\hbox to \myboxwidth{%
$\Pi_{22,31}$ spans $L_{251.3}$%
\hfil
$3222622322622262222622$%
}%
\box\matricesbox
\else
\hbox to \myboxwidth{%
\unhbox\onelinebox
}%
\fi
\else
\hbox to \myboxwidth{%
$\Pi_{22,31}$ spans $L_{251.3}$%
\hfil}%
\hbox to \myboxwidth{%
$3222622322622262222622$%
\hfil}%
\box\matricesbox
\fi
}%
\hfill\discretionary{}{}{}%
\setbox\matricesbox=\hbox{%
{$\left[\!\llap{\phantom{%
\begingroup \smaller\smaller\smaller
\endgroup%
}}\!\right]$}%
}%
\ifdim\wd\matricesbox>\halfwidth\myboxwidth=\hsize\else\myboxwidth=\halfwidth\fi
\vbox{%
\ifdim\myboxwidth=\hsize
\setbox\onelinebox=\hbox{%
\vbox{\hbox{%
$\Pi_{22,32}$ spans $L_{251.3}$%
}\hbox{%
$3222622322622262226222$%
}%
}%
\hfill\copy\matricesbox
}%
\ifdim\wd\onelinebox>\myboxwidth
\hbox to \myboxwidth{%
$\Pi_{22,32}$ spans $L_{251.3}$%
\hfil
$3222622322622262226222$%
}%
\box\matricesbox
\else
\hbox to \myboxwidth{%
\unhbox\onelinebox
}%
\fi
\else
\hbox to \myboxwidth{%
$\Pi_{22,32}$ spans $L_{251.3}$%
\hfil}%
\hbox to \myboxwidth{%
$3222622322622262226222$%
\hfil}%
\box\matricesbox
\fi
}%
\hfill\discretionary{}{}{}%
\setbox\matricesbox=\hbox{%
{$\left[\!\llap{\phantom{%
\begingroup \smaller\smaller\smaller
\endgroup%
}}\!\right]$}%
}%
\ifdim\wd\matricesbox>\halfwidth\myboxwidth=\hsize\else\myboxwidth=\halfwidth\fi
\vbox{%
\ifdim\myboxwidth=\hsize
\setbox\onelinebox=\hbox{%
\vbox{\hbox{%
$\Pi_{22,33}$ spans $L_{251.3}$%
}\hbox{%
$3222622262232226226222$%
}%
}%
\hfill\copy\matricesbox
}%
\ifdim\wd\onelinebox>\myboxwidth
\hbox to \myboxwidth{%
$\Pi_{22,33}$ spans $L_{251.3}$%
\hfil
$3222622262232226226222$%
}%
\box\matricesbox
\else
\hbox to \myboxwidth{%
\unhbox\onelinebox
}%
\fi
\else
\hbox to \myboxwidth{%
$\Pi_{22,33}$ spans $L_{251.3}$%
\hfil}%
\hbox to \myboxwidth{%
$3222622262232226226222$%
\hfil}%
\box\matricesbox
\fi
}%
\hfill\discretionary{}{}{}%
\setbox\matricesbox=\hbox{%
{$\left[\!\llap{\phantom{%
\begingroup \smaller\smaller\smaller
\endgroup%
}}\!\right]$}%
}%
\ifdim\wd\matricesbox>\halfwidth\myboxwidth=\hsize\else\myboxwidth=\halfwidth\fi
\vbox{%
\ifdim\myboxwidth=\hsize
\setbox\onelinebox=\hbox{%
\vbox{\hbox{%
$\Pi_{22,34}$ spans $L_{251.3}$%
}\hbox{%
$3222622262232262222622$%
}%
}%
\hfill\copy\matricesbox
}%
\ifdim\wd\onelinebox>\myboxwidth
\hbox to \myboxwidth{%
$\Pi_{22,34}$ spans $L_{251.3}$%
\hfil
$3222622262232262222622$%
}%
\box\matricesbox
\else
\hbox to \myboxwidth{%
\unhbox\onelinebox
}%
\fi
\else
\hbox to \myboxwidth{%
$\Pi_{22,34}$ spans $L_{251.3}$%
\hfil}%
\hbox to \myboxwidth{%
$3222622262232262222622$%
\hfil}%
\box\matricesbox
\fi
}%
\hfill\discretionary{}{}{}%

\vskip2pt\hrule\vskip2pt

\leavevmode\setbox\matricesbox=\hbox{%
{$\left[\!\llap{\phantom{%
\begingroup \smaller\smaller\smaller
\endgroup%
}}\!\right]$}%
}%
\ifdim\wd\matricesbox>\halfwidth\myboxwidth=\hsize\else\myboxwidth=\halfwidth\fi
\vbox{%
\ifdim\myboxwidth=\hsize
\setbox\onelinebox=\hbox{%
\vbox{\hbox{%
$\Pi_{23,1}$ spans $L_{16.9}$%
}\hbox{%
$3632|23636363636\slashthree6363636\rtimes D_{2}$%
}%
}%
\hfill\copy\matricesbox
}%
\ifdim\wd\onelinebox>\myboxwidth
\hbox to \myboxwidth{%
$\Pi_{23,1}$ spans $L_{16.9}$%
\hfil
$3632|23636363636\slashthree6363636\rtimes D_{2}$%
}%
\box\matricesbox
\else
\hbox to \myboxwidth{%
\unhbox\onelinebox
}%
\fi
\else
\hbox to \myboxwidth{%
$\Pi_{23,1}$ spans $L_{16.9}$%
\hfil}%
\hbox to \myboxwidth{%
$3632|23636363636\slashthree6363636\rtimes D_{2}$%
\hfil}%
\box\matricesbox
\fi
}%
\hfill\discretionary{}{}{}%
\setbox\matricesbox=\hbox{%
{$\left[\!\llap{\phantom{%
\begingroup \smaller\smaller\smaller
\endgroup%
}}\!\right]$}%
}%
\ifdim\wd\matricesbox>\halfwidth\myboxwidth=\hsize\else\myboxwidth=\halfwidth\fi
\vbox{%
\ifdim\myboxwidth=\hsize
\setbox\onelinebox=\hbox{%
\vbox{\hbox{%
$\Pi_{23,2}$ spans $L_{16.7}$%
}\hbox{%
$363632|23636363636\slashthree63636\rtimes D_{2}$%
}%
}%
\hfill\copy\matricesbox
}%
\ifdim\wd\onelinebox>\myboxwidth
\hbox to \myboxwidth{%
$\Pi_{23,2}$ spans $L_{16.7}$%
\hfil
$363632|23636363636\slashthree63636\rtimes D_{2}$%
}%
\box\matricesbox
\else
\hbox to \myboxwidth{%
\unhbox\onelinebox
}%
\fi
\else
\hbox to \myboxwidth{%
$\Pi_{23,2}$ spans $L_{16.7}$%
\hfil}%
\hbox to \myboxwidth{%
$363632|23636363636\slashthree63636\rtimes D_{2}$%
\hfil}%
\box\matricesbox
\fi
}%
\hfill\discretionary{}{}{}%
\setbox\matricesbox=\hbox{%
{$\left[\!\llap{\phantom{%
\begingroup \smaller\smaller\smaller
\endgroup%
}}\!\right]$}%
}%
\ifdim\wd\matricesbox>\halfwidth\myboxwidth=\hsize\else\myboxwidth=\halfwidth\fi
\vbox{%
\ifdim\myboxwidth=\hsize
\setbox\onelinebox=\hbox{%
\vbox{\hbox{%
$\Pi_{23,3}$ spans $L_{251.3}$%
}\hbox{%
$32226222622262226222622$%
}%
}%
\hfill\copy\matricesbox
}%
\ifdim\wd\onelinebox>\myboxwidth
\hbox to \myboxwidth{%
$\Pi_{23,3}$ spans $L_{251.3}$%
\hfil
$32226222622262226222622$%
}%
\box\matricesbox
\else
\hbox to \myboxwidth{%
\unhbox\onelinebox
}%
\fi
\else
\hbox to \myboxwidth{%
$\Pi_{23,3}$ spans $L_{251.3}$%
\hfil}%
\hbox to \myboxwidth{%
$32226222622262226222622$%
\hfil}%
\box\matricesbox
\fi
}%
\hfill\discretionary{}{}{}%
\setbox\matricesbox=\hbox{%
{$\left[\!\llap{\phantom{%
\begingroup \smaller\smaller\smaller
\endgroup%
}}\!\right]$}%
}%
\ifdim\wd\matricesbox>\halfwidth\myboxwidth=\hsize\else\myboxwidth=\halfwidth\fi
\vbox{%
\ifdim\myboxwidth=\hsize
\setbox\onelinebox=\hbox{%
\vbox{\hbox{%
$\Pi_{23,4}$ spans $L_{251.3}$%
}\hbox{%
$32226222622262226226222$%
}%
}%
\hfill\copy\matricesbox
}%
\ifdim\wd\onelinebox>\myboxwidth
\hbox to \myboxwidth{%
$\Pi_{23,4}$ spans $L_{251.3}$%
\hfil
$32226222622262226226222$%
}%
\box\matricesbox
\else
\hbox to \myboxwidth{%
\unhbox\onelinebox
}%
\fi
\else
\hbox to \myboxwidth{%
$\Pi_{23,4}$ spans $L_{251.3}$%
\hfil}%
\hbox to \myboxwidth{%
$32226222622262226226222$%
\hfil}%
\box\matricesbox
\fi
}%
\hfill\discretionary{}{}{}%
\setbox\matricesbox=\hbox{%
{$\left[\!\llap{\phantom{%
\begingroup \smaller\smaller\smaller
\endgroup%
}}\!\right]$}%
}%
\ifdim\wd\matricesbox>\halfwidth\myboxwidth=\hsize\else\myboxwidth=\halfwidth\fi
\vbox{%
\ifdim\myboxwidth=\hsize
\setbox\onelinebox=\hbox{%
\vbox{\hbox{%
$\Pi_{23,5}$ spans $L_{251.3}$%
}\hbox{%
$32226222622262262222622$%
}%
}%
\hfill\copy\matricesbox
}%
\ifdim\wd\onelinebox>\myboxwidth
\hbox to \myboxwidth{%
$\Pi_{23,5}$ spans $L_{251.3}$%
\hfil
$32226222622262262222622$%
}%
\box\matricesbox
\else
\hbox to \myboxwidth{%
\unhbox\onelinebox
}%
\fi
\else
\hbox to \myboxwidth{%
$\Pi_{23,5}$ spans $L_{251.3}$%
\hfil}%
\hbox to \myboxwidth{%
$32226222622262262222622$%
\hfil}%
\box\matricesbox
\fi
}%
\hfill\discretionary{}{}{}%
\setbox\matricesbox=\hbox{%
{$\left[\!\llap{\phantom{%
\begingroup \smaller\smaller\smaller
\endgroup%
}}\!\right]$}%
}%
\ifdim\wd\matricesbox>\halfwidth\myboxwidth=\hsize\else\myboxwidth=\halfwidth\fi
\vbox{%
\ifdim\myboxwidth=\hsize
\setbox\onelinebox=\hbox{%
\vbox{\hbox{%
$\Pi_{23,6}$ spans $L_{251.3}$%
}\hbox{%
$32226222622262262226222$%
}%
}%
\hfill\copy\matricesbox
}%
\ifdim\wd\onelinebox>\myboxwidth
\hbox to \myboxwidth{%
$\Pi_{23,6}$ spans $L_{251.3}$%
\hfil
$32226222622262262226222$%
}%
\box\matricesbox
\else
\hbox to \myboxwidth{%
\unhbox\onelinebox
}%
\fi
\else
\hbox to \myboxwidth{%
$\Pi_{23,6}$ spans $L_{251.3}$%
\hfil}%
\hbox to \myboxwidth{%
$32226222622262262226222$%
\hfil}%
\box\matricesbox
\fi
}%
\hfill\discretionary{}{}{}%
\setbox\matricesbox=\hbox{%
{$\left[\!\llap{\phantom{%
\begingroup \smaller\smaller\smaller
\endgroup%
}}\!\right]$}%
}%
\ifdim\wd\matricesbox>\halfwidth\myboxwidth=\hsize\else\myboxwidth=\halfwidth\fi
\vbox{%
\ifdim\myboxwidth=\hsize
\setbox\onelinebox=\hbox{%
\vbox{\hbox{%
$\Pi_{23,7}$ spans $L_{251.3}$%
}\hbox{%
$32226222622622226222622$%
}%
}%
\hfill\copy\matricesbox
}%
\ifdim\wd\onelinebox>\myboxwidth
\hbox to \myboxwidth{%
$\Pi_{23,7}$ spans $L_{251.3}$%
\hfil
$32226222622622226222622$%
}%
\box\matricesbox
\else
\hbox to \myboxwidth{%
\unhbox\onelinebox
}%
\fi
\else
\hbox to \myboxwidth{%
$\Pi_{23,7}$ spans $L_{251.3}$%
\hfil}%
\hbox to \myboxwidth{%
$32226222622622226222622$%
\hfil}%
\box\matricesbox
\fi
}%
\hfill\discretionary{}{}{}%
\setbox\matricesbox=\hbox{%
{$\left[\!\llap{\phantom{%
\begingroup \smaller\smaller\smaller
\endgroup%
}}\!\right]$}%
}%
\ifdim\wd\matricesbox>\halfwidth\myboxwidth=\hsize\else\myboxwidth=\halfwidth\fi
\vbox{%
\ifdim\myboxwidth=\hsize
\setbox\onelinebox=\hbox{%
\vbox{\hbox{%
$\Pi_{23,8}$ spans $L_{251.3}$%
}\hbox{%
$32226222622622226226222$%
}%
}%
\hfill\copy\matricesbox
}%
\ifdim\wd\onelinebox>\myboxwidth
\hbox to \myboxwidth{%
$\Pi_{23,8}$ spans $L_{251.3}$%
\hfil
$32226222622622226226222$%
}%
\box\matricesbox
\else
\hbox to \myboxwidth{%
\unhbox\onelinebox
}%
\fi
\else
\hbox to \myboxwidth{%
$\Pi_{23,8}$ spans $L_{251.3}$%
\hfil}%
\hbox to \myboxwidth{%
$32226222622622226226222$%
\hfil}%
\box\matricesbox
\fi
}%
\hfill\discretionary{}{}{}%
\setbox\matricesbox=\hbox{%
{$\left[\!\llap{\phantom{%
\begingroup \smaller\smaller\smaller
\endgroup%
}}\!\right]$}%
}%
\ifdim\wd\matricesbox>\halfwidth\myboxwidth=\hsize\else\myboxwidth=\halfwidth\fi
\vbox{%
\ifdim\myboxwidth=\hsize
\setbox\onelinebox=\hbox{%
\vbox{\hbox{%
$\Pi_{23,9}$ spans $L_{251.3}$%
}\hbox{%
$32226222622622262222622$%
}%
}%
\hfill\copy\matricesbox
}%
\ifdim\wd\onelinebox>\myboxwidth
\hbox to \myboxwidth{%
$\Pi_{23,9}$ spans $L_{251.3}$%
\hfil
$32226222622622262222622$%
}%
\box\matricesbox
\else
\hbox to \myboxwidth{%
\unhbox\onelinebox
}%
\fi
\else
\hbox to \myboxwidth{%
$\Pi_{23,9}$ spans $L_{251.3}$%
\hfil}%
\hbox to \myboxwidth{%
$32226222622622262222622$%
\hfil}%
\box\matricesbox
\fi
}%
\hfill\discretionary{}{}{}%
\setbox\matricesbox=\hbox{%
{$\left[\!\llap{\phantom{%
\begingroup \smaller\smaller\smaller
\endgroup%
}}\!\right]$}%
}%
\ifdim\wd\matricesbox>\halfwidth\myboxwidth=\hsize\else\myboxwidth=\halfwidth\fi
\vbox{%
\ifdim\myboxwidth=\hsize
\setbox\onelinebox=\hbox{%
\vbox{\hbox{%
$\Pi_{23,10}$ spans $L_{251.3}$%
}\hbox{%
$32226226222262226222622$%
}%
}%
\hfill\copy\matricesbox
}%
\ifdim\wd\onelinebox>\myboxwidth
\hbox to \myboxwidth{%
$\Pi_{23,10}$ spans $L_{251.3}$%
\hfil
$32226226222262226222622$%
}%
\box\matricesbox
\else
\hbox to \myboxwidth{%
\unhbox\onelinebox
}%
\fi
\else
\hbox to \myboxwidth{%
$\Pi_{23,10}$ spans $L_{251.3}$%
\hfil}%
\hbox to \myboxwidth{%
$32226226222262226222622$%
\hfil}%
\box\matricesbox
\fi
}%
\hfill\discretionary{}{}{}%
\setbox\matricesbox=\hbox{%
{$\left[\!\llap{\phantom{%
\begingroup \smaller\smaller\smaller
\endgroup%
}}\!\right]$}%
}%
\ifdim\wd\matricesbox>\halfwidth\myboxwidth=\hsize\else\myboxwidth=\halfwidth\fi
\vbox{%
\ifdim\myboxwidth=\hsize
\setbox\onelinebox=\hbox{%
\vbox{\hbox{%
$\Pi_{23,11}$ spans $L_{251.3}$%
}\hbox{%
$32226226222262226226222$%
}%
}%
\hfill\copy\matricesbox
}%
\ifdim\wd\onelinebox>\myboxwidth
\hbox to \myboxwidth{%
$\Pi_{23,11}$ spans $L_{251.3}$%
\hfil
$32226226222262226226222$%
}%
\box\matricesbox
\else
\hbox to \myboxwidth{%
\unhbox\onelinebox
}%
\fi
\else
\hbox to \myboxwidth{%
$\Pi_{23,11}$ spans $L_{251.3}$%
\hfil}%
\hbox to \myboxwidth{%
$32226226222262226226222$%
\hfil}%
\box\matricesbox
\fi
}%
\hfill\discretionary{}{}{}%
\setbox\matricesbox=\hbox{%
{$\left[\!\llap{\phantom{%
\begingroup \smaller\smaller\smaller
\endgroup%
}}\!\right]$}%
}%
\ifdim\wd\matricesbox>\halfwidth\myboxwidth=\hsize\else\myboxwidth=\halfwidth\fi
\vbox{%
\ifdim\myboxwidth=\hsize
\setbox\onelinebox=\hbox{%
\vbox{\hbox{%
$\Pi_{23,12}$ spans $L_{251.3}$%
}\hbox{%
$32226226222262262222622$%
}%
}%
\hfill\copy\matricesbox
}%
\ifdim\wd\onelinebox>\myboxwidth
\hbox to \myboxwidth{%
$\Pi_{23,12}$ spans $L_{251.3}$%
\hfil
$32226226222262262222622$%
}%
\box\matricesbox
\else
\hbox to \myboxwidth{%
\unhbox\onelinebox
}%
\fi
\else
\hbox to \myboxwidth{%
$\Pi_{23,12}$ spans $L_{251.3}$%
\hfil}%
\hbox to \myboxwidth{%
$32226226222262262222622$%
\hfil}%
\box\matricesbox
\fi
}%
\hfill\discretionary{}{}{}%
\setbox\matricesbox=\hbox{%
{$\left[\!\llap{\phantom{%
\begingroup \smaller\smaller\smaller
\endgroup%
}}\!\right]$}%
}%
\ifdim\wd\matricesbox>\halfwidth\myboxwidth=\hsize\else\myboxwidth=\halfwidth\fi
\vbox{%
\ifdim\myboxwidth=\hsize
\setbox\onelinebox=\hbox{%
\vbox{\hbox{%
$\Pi_{23,13}$ spans $L_{251.3}$%
}\hbox{%
$32226226222622226222622$%
}%
}%
\hfill\copy\matricesbox
}%
\ifdim\wd\onelinebox>\myboxwidth
\hbox to \myboxwidth{%
$\Pi_{23,13}$ spans $L_{251.3}$%
\hfil
$32226226222622226222622$%
}%
\box\matricesbox
\else
\hbox to \myboxwidth{%
\unhbox\onelinebox
}%
\fi
\else
\hbox to \myboxwidth{%
$\Pi_{23,13}$ spans $L_{251.3}$%
\hfil}%
\hbox to \myboxwidth{%
$32226226222622226222622$%
\hfil}%
\box\matricesbox
\fi
}%
\hfill\discretionary{}{}{}%
\setbox\matricesbox=\hbox{%
{$\left[\!\llap{\phantom{%
\begingroup \smaller\smaller\smaller
\endgroup%
}}\!\right]$}%
}%
\ifdim\wd\matricesbox>\halfwidth\myboxwidth=\hsize\else\myboxwidth=\halfwidth\fi
\vbox{%
\ifdim\myboxwidth=\hsize
\setbox\onelinebox=\hbox{%
\vbox{\hbox{%
$\Pi_{23,14}$ spans $L_{251.3}$%
}\hbox{%
$32226226222622262222622$%
}%
}%
\hfill\copy\matricesbox
}%
\ifdim\wd\onelinebox>\myboxwidth
\hbox to \myboxwidth{%
$\Pi_{23,14}$ spans $L_{251.3}$%
\hfil
$32226226222622262222622$%
}%
\box\matricesbox
\else
\hbox to \myboxwidth{%
\unhbox\onelinebox
}%
\fi
\else
\hbox to \myboxwidth{%
$\Pi_{23,14}$ spans $L_{251.3}$%
\hfil}%
\hbox to \myboxwidth{%
$32226226222622262222622$%
\hfil}%
\box\matricesbox
\fi
}%
\hfill\discretionary{}{}{}%
\setbox\matricesbox=\hbox{%
{$\left[\!\llap{\phantom{%
\begingroup \smaller\smaller\smaller
\endgroup%
}}\!\right]$}%
}%
\ifdim\wd\matricesbox>\halfwidth\myboxwidth=\hsize\else\myboxwidth=\halfwidth\fi
\vbox{%
\ifdim\myboxwidth=\hsize
\setbox\onelinebox=\hbox{%
\vbox{\hbox{%
$\Pi_{23,15}$ spans $L_{251.3}$%
}\hbox{%
$32262222622262226222622$%
}%
}%
\hfill\copy\matricesbox
}%
\ifdim\wd\onelinebox>\myboxwidth
\hbox to \myboxwidth{%
$\Pi_{23,15}$ spans $L_{251.3}$%
\hfil
$32262222622262226222622$%
}%
\box\matricesbox
\else
\hbox to \myboxwidth{%
\unhbox\onelinebox
}%
\fi
\else
\hbox to \myboxwidth{%
$\Pi_{23,15}$ spans $L_{251.3}$%
\hfil}%
\hbox to \myboxwidth{%
$32262222622262226222622$%
\hfil}%
\box\matricesbox
\fi
}%
\hfill\discretionary{}{}{}%
\setbox\matricesbox=\hbox{%
{$\left[\!\llap{\phantom{%
\begingroup \smaller\smaller\smaller
\endgroup%
}}\!\right]$}%
}%
\ifdim\wd\matricesbox>\halfwidth\myboxwidth=\hsize\else\myboxwidth=\halfwidth\fi
\vbox{%
\ifdim\myboxwidth=\hsize
\setbox\onelinebox=\hbox{%
\vbox{\hbox{%
$\Pi_{23,16}$ spans $L_{251.3}$%
}\hbox{%
$32262222622262262222622$%
}%
}%
\hfill\copy\matricesbox
}%
\ifdim\wd\onelinebox>\myboxwidth
\hbox to \myboxwidth{%
$\Pi_{23,16}$ spans $L_{251.3}$%
\hfil
$32262222622262262222622$%
}%
\box\matricesbox
\else
\hbox to \myboxwidth{%
\unhbox\onelinebox
}%
\fi
\else
\hbox to \myboxwidth{%
$\Pi_{23,16}$ spans $L_{251.3}$%
\hfil}%
\hbox to \myboxwidth{%
$32262222622262262222622$%
\hfil}%
\box\matricesbox
\fi
}%
\hfill\discretionary{}{}{}%
\setbox\matricesbox=\hbox{%
{$\left[\!\llap{\phantom{%
\begingroup \smaller\smaller\smaller
\endgroup%
}}\!\right]$}%
}%
\ifdim\wd\matricesbox>\halfwidth\myboxwidth=\hsize\else\myboxwidth=\halfwidth\fi
\vbox{%
\ifdim\myboxwidth=\hsize
\setbox\onelinebox=\hbox{%
\vbox{\hbox{%
$\Pi_{23,17}$ spans $L_{251.3}$%
}\hbox{%
$32262222622622226222622$%
}%
}%
\hfill\copy\matricesbox
}%
\ifdim\wd\onelinebox>\myboxwidth
\hbox to \myboxwidth{%
$\Pi_{23,17}$ spans $L_{251.3}$%
\hfil
$32262222622622226222622$%
}%
\box\matricesbox
\else
\hbox to \myboxwidth{%
\unhbox\onelinebox
}%
\fi
\else
\hbox to \myboxwidth{%
$\Pi_{23,17}$ spans $L_{251.3}$%
\hfil}%
\hbox to \myboxwidth{%
$32262222622622226222622$%
\hfil}%
\box\matricesbox
\fi
}%
\hfill\discretionary{}{}{}%
\setbox\matricesbox=\hbox{%
{$\left[\!\llap{\phantom{%
\begingroup \smaller\smaller\smaller
\endgroup%
}}\!\right]$}%
}%
\ifdim\wd\matricesbox>\halfwidth\myboxwidth=\hsize\else\myboxwidth=\halfwidth\fi
\vbox{%
\ifdim\myboxwidth=\hsize
\setbox\onelinebox=\hbox{%
\vbox{\hbox{%
$\Pi_{23,18}$ spans $L_{251.3}$%
}\hbox{%
$32262226222262226222622$%
}%
}%
\hfill\copy\matricesbox
}%
\ifdim\wd\onelinebox>\myboxwidth
\hbox to \myboxwidth{%
$\Pi_{23,18}$ spans $L_{251.3}$%
\hfil
$32262226222262226222622$%
}%
\box\matricesbox
\else
\hbox to \myboxwidth{%
\unhbox\onelinebox
}%
\fi
\else
\hbox to \myboxwidth{%
$\Pi_{23,18}$ spans $L_{251.3}$%
\hfil}%
\hbox to \myboxwidth{%
$32262226222262226222622$%
\hfil}%
\box\matricesbox
\fi
}%
\hfill\discretionary{}{}{}%

\vskip2pt\hrule\vskip2pt

\leavevmode\setbox\matricesbox=\hbox{%
{$\left[\!\llap{\phantom{%
\begingroup \smaller\smaller\smaller\begin{tabular}{@{}c@{}}%
\phantom{0}\\\phantom{0}\\\phantom{0}\\\phantom{0}
\end{tabular}\endgroup%
}}\right.$}%
\begingroup \smaller\smaller\smaller\begin{tabular}{@{}c@{}}%
-15\\\phantom{0}\\\phantom{0}\\\phantom{0}
\end{tabular}\endgroup%
\kern3pt%
\begingroup \smaller\smaller\smaller\begin{tabular}{@{}c@{}}%
\phantom{0}\\1\\\phantom{0}\\\phantom{0}
\end{tabular}\endgroup%
\kern3pt%
\begingroup \smaller\smaller\smaller\begin{tabular}{@{}c@{}}%
\phantom{0}\\\phantom{0}\\1\\\phantom{0}
\end{tabular}\endgroup%
\kern3pt%
\begingroup \smaller\smaller\smaller\begin{tabular}{@{}c@{}}%
\phantom{0}\\\phantom{0}\\\phantom{0}\\1
\end{tabular}\endgroup%
{$\left.\llap{\phantom{%
\begingroup \smaller\smaller\smaller\begin{tabular}{@{}c@{}}%
\phantom{0}\\\phantom{0}\\\phantom{0}\\\phantom{0}
\end{tabular}\endgroup%
}}\!\right]$}%
{$\left[\!\llap{\phantom{%
\begingroup \smaller\smaller\smaller\begin{tabular}{@{}c@{}}%
0\\0\\0\\0
\end{tabular}\endgroup%
}}\right.$}%
\begingroup \smaller\smaller\smaller\begin{tabular}{@{}c@{}}%
6\\19\\-8\\-11
\end{tabular}\endgroup%
\kern3pt%
\begingroup \smaller\smaller\smaller\begin{tabular}{@{}c@{}}%
2\\6\\-1\\-5
\end{tabular}\endgroup%
{$\left.\llap{\phantom{%
\begingroup \smaller\smaller\smaller\begin{tabular}{@{}c@{}}%
0\\0\\0\\0
\end{tabular}\endgroup%
}}\!\right]$}%
}%
\ifdim\wd\matricesbox>\halfwidth\myboxwidth=\hsize\else\myboxwidth=\halfwidth\fi
\vbox{%
\ifdim\myboxwidth=\hsize
\setbox\onelinebox=\hbox{%
\vbox{\hbox{%
$\Pi_{24,1}$ spans $L_{16.8}$%
}\hbox{%
$\slashthree6\slashthree6\slashthree6\slashthree6\slashthree6\slashthree6\slashthree6\slashthree6\slashthree6\slashthree6\slashthree6\slashthree6\rtimes D_{12}$%
}%
}%
\hfill\copy\matricesbox
}%
\ifdim\wd\onelinebox>\myboxwidth
\hbox to \myboxwidth{%
$\Pi_{24,1}$ spans $L_{16.8}$%
\hfil
$\slashthree6\slashthree6\slashthree6\slashthree6\slashthree6\slashthree6\slashthree6\slashthree6\slashthree6\slashthree6\slashthree6\slashthree6\rtimes D_{12}$%
}%
\box\matricesbox
\else
\hbox to \myboxwidth{%
\unhbox\onelinebox
}%
\fi
\else
\hbox to \myboxwidth{%
$\Pi_{24,1}$ spans $L_{16.8}$%
\hfil}%
\hbox to \myboxwidth{%
$\slashthree6\slashthree6\slashthree6\slashthree6\slashthree6\slashthree6\slashthree6\slashthree6\slashthree6\slashthree6\slashthree6\slashthree6\rtimes D_{12}$%
\hfil}%
\box\matricesbox
\fi
}%
\hfill\discretionary{}{}{}%
\setbox\matricesbox=\hbox{%
{$\left[\!\llap{\phantom{%
\begingroup \smaller\smaller\smaller\begin{tabular}{@{}c@{}}%
\phantom{0}\\\phantom{0}\\\phantom{0}\\\phantom{0}
\end{tabular}\endgroup%
}}\right.$}%
\begingroup \smaller\smaller\smaller\begin{tabular}{@{}c@{}}%
-5\\\phantom{0}\\\phantom{0}\\\phantom{0}
\end{tabular}\endgroup%
\kern3pt%
\begingroup \smaller\smaller\smaller\begin{tabular}{@{}c@{}}%
\phantom{0}\\2\\\phantom{0}\\\phantom{0}
\end{tabular}\endgroup%
\kern3pt%
\begingroup \smaller\smaller\smaller\begin{tabular}{@{}c@{}}%
\phantom{0}\\\phantom{0}\\2\\\phantom{0}
\end{tabular}\endgroup%
\kern3pt%
\begingroup \smaller\smaller\smaller\begin{tabular}{@{}c@{}}%
\phantom{0}\\\phantom{0}\\\phantom{0}\\2
\end{tabular}\endgroup%
{$\left.\llap{\phantom{%
\begingroup \smaller\smaller\smaller\begin{tabular}{@{}c@{}}%
\phantom{0}\\\phantom{0}\\\phantom{0}\\\phantom{0}
\end{tabular}\endgroup%
}}\!\right]$}%
{$\left[\!\llap{\phantom{%
\begingroup \smaller\smaller\smaller\begin{tabular}{@{}c@{}}%
0\\0\\0\\0
\end{tabular}\endgroup%
}}\right.$}%
\begingroup \smaller\smaller\smaller\begin{tabular}{@{}c@{}}%
3\\-4\\2\\2
\end{tabular}\endgroup%
\kern3pt%
\begingroup \smaller\smaller\smaller\begin{tabular}{@{}c@{}}%
36\\-45\\12\\33
\end{tabular}\endgroup%
\kern3pt%
\begingroup \smaller\smaller\smaller\begin{tabular}{@{}c@{}}%
12\\-14\\1\\13
\end{tabular}\endgroup%
\kern3pt%
\begingroup \smaller\smaller\smaller\begin{tabular}{@{}c@{}}%
4\\-4\\-1\\5
\end{tabular}\endgroup%
\kern3pt%
\begingroup \smaller\smaller\smaller\begin{tabular}{@{}c@{}}%
3\\-2\\-2\\4
\end{tabular}\endgroup%
{$\left.\llap{\phantom{%
\begingroup \smaller\smaller\smaller\begin{tabular}{@{}c@{}}%
0\\0\\0\\0
\end{tabular}\endgroup%
}}\!\right]$}%
}%
\ifdim\wd\matricesbox>\halfwidth\myboxwidth=\hsize\else\myboxwidth=\halfwidth\fi
\vbox{%
\ifdim\myboxwidth=\hsize
\setbox\onelinebox=\hbox{%
\vbox{\hbox{%
$\Pi_{24,2}$ spans $L_{251.3}$%
}\hbox{%
$62|2622|2262|2622|2262|2622|22\rtimes D_{6}$%
}%
}%
\hfill\copy\matricesbox
}%
\ifdim\wd\onelinebox>\myboxwidth
\hbox to \myboxwidth{%
$\Pi_{24,2}$ spans $L_{251.3}$%
\hfil
$62|2622|2262|2622|2262|2622|22\rtimes D_{6}$%
}%
\box\matricesbox
\else
\hbox to \myboxwidth{%
\unhbox\onelinebox
}%
\fi
\else
\hbox to \myboxwidth{%
$\Pi_{24,2}$ spans $L_{251.3}$%
\hfil}%
\hbox to \myboxwidth{%
$62|2622|2262|2622|2262|2622|22\rtimes D_{6}$%
\hfil}%
\box\matricesbox
\fi
}%
\hfill\discretionary{}{}{}%
\setbox\matricesbox=\hbox{%
{$\left[\!\llap{\phantom{%
\begingroup \smaller\smaller\smaller\begin{tabular}{@{}c@{}}%
\phantom{0}\\\phantom{0}\\\phantom{0}\\\phantom{0}
\end{tabular}\endgroup%
}}\right.$}%
\begingroup \smaller\smaller\smaller\begin{tabular}{@{}c@{}}%
-5\\\phantom{0}\\\phantom{0}\\\phantom{0}
\end{tabular}\endgroup%
\kern3pt%
\begingroup \smaller\smaller\smaller\begin{tabular}{@{}c@{}}%
\phantom{0}\\6\\\phantom{0}\\\phantom{0}
\end{tabular}\endgroup%
\kern3pt%
\begingroup \smaller\smaller\smaller\begin{tabular}{@{}c@{}}%
\phantom{0}\\\phantom{0}\\6\\\phantom{0}
\end{tabular}\endgroup%
\kern3pt%
\begingroup \smaller\smaller\smaller\begin{tabular}{@{}c@{}}%
\phantom{0}\\\phantom{0}\\\phantom{0}\\6
\end{tabular}\endgroup%
{$\left.\llap{\phantom{%
\begingroup \smaller\smaller\smaller\begin{tabular}{@{}c@{}}%
\phantom{0}\\\phantom{0}\\\phantom{0}\\\phantom{0}
\end{tabular}\endgroup%
}}\!\right]$}%
{$\left[\!\llap{\phantom{%
\begingroup \smaller\smaller\smaller\begin{tabular}{@{}c@{}}%
0\\0\\0\\0
\end{tabular}\endgroup%
}}\right.$}%
\begingroup \smaller\smaller\smaller\begin{tabular}{@{}c@{}}%
12\\-5\\9\\-4
\end{tabular}\endgroup%
\kern3pt%
\begingroup \smaller\smaller\smaller\begin{tabular}{@{}c@{}}%
36\\-19\\26\\-7
\end{tabular}\endgroup%
\kern3pt%
\begingroup \smaller\smaller\smaller\begin{tabular}{@{}c@{}}%
3\\-2\\2\\0
\end{tabular}\endgroup%
\kern3pt%
\begingroup \smaller\smaller\smaller\begin{tabular}{@{}c@{}}%
4\\-3\\2\\1
\end{tabular}\endgroup%
{$\left.\llap{\phantom{%
\begingroup \smaller\smaller\smaller\begin{tabular}{@{}c@{}}%
0\\0\\0\\0
\end{tabular}\endgroup%
}}\!\right]$}%
}%
\ifdim\wd\matricesbox>\halfwidth\myboxwidth=\hsize\else\myboxwidth=\halfwidth\fi
\vbox{%
\ifdim\myboxwidth=\hsize
\setbox\onelinebox=\hbox{%
\vbox{\hbox{%
$\Pi_{24,3}$ spans $L_{251.3}$%
}\hbox{%
$622262226222622262226222\rtimes C_{6}$%
}%
}%
\hfill\copy\matricesbox
}%
\ifdim\wd\onelinebox>\myboxwidth
\hbox to \myboxwidth{%
$\Pi_{24,3}$ spans $L_{251.3}$%
\hfil
$622262226222622262226222\rtimes C_{6}$%
}%
\box\matricesbox
\else
\hbox to \myboxwidth{%
\unhbox\onelinebox
}%
\fi
\else
\hbox to \myboxwidth{%
$\Pi_{24,3}$ spans $L_{251.3}$%
\hfil}%
\hbox to \myboxwidth{%
$622262226222622262226222\rtimes C_{6}$%
\hfil}%
\box\matricesbox
\fi
}%
\hfill\discretionary{}{}{}%
\setbox\matricesbox=\hbox{%
{$\left[\!\llap{\phantom{%
\begingroup \smaller\smaller\smaller
\endgroup%
}}\!\right]$}%
}%
\ifdim\wd\matricesbox>\halfwidth\myboxwidth=\hsize\else\myboxwidth=\halfwidth\fi
\vbox{%
\ifdim\myboxwidth=\hsize
\setbox\onelinebox=\hbox{%
\vbox{\hbox{%
$\Pi_{24,4}$ spans $L_{251.3}$%
}\hbox{%
$6222622262|262226222622|22\rtimes D_{2}$%
}%
}%
\hfill\copy\matricesbox
}%
\ifdim\wd\onelinebox>\myboxwidth
\hbox to \myboxwidth{%
$\Pi_{24,4}$ spans $L_{251.3}$%
\hfil
$6222622262|262226222622|22\rtimes D_{2}$%
}%
\box\matricesbox
\else
\hbox to \myboxwidth{%
\unhbox\onelinebox
}%
\fi
\else
\hbox to \myboxwidth{%
$\Pi_{24,4}$ spans $L_{251.3}$%
\hfil}%
\hbox to \myboxwidth{%
$6222622262|262226222622|22\rtimes D_{2}$%
\hfil}%
\box\matricesbox
\fi
}%
\hfill\discretionary{}{}{}%
\setbox\matricesbox=\hbox{%
{$\left[\!\llap{\phantom{%
\begingroup \smaller\smaller\smaller
\endgroup%
}}\!\right]$}%
}%
\ifdim\wd\matricesbox>\halfwidth\myboxwidth=\hsize\else\myboxwidth=\halfwidth\fi
\vbox{%
\ifdim\myboxwidth=\hsize
\setbox\onelinebox=\hbox{%
\vbox{\hbox{%
$\Pi_{24,5}$ spans $L_{251.3}$%
}\hbox{%
$6222622622|226226222622|22\rtimes D_{2}$%
}%
}%
\hfill\copy\matricesbox
}%
\ifdim\wd\onelinebox>\myboxwidth
\hbox to \myboxwidth{%
$\Pi_{24,5}$ spans $L_{251.3}$%
\hfil
$6222622622|226226222622|22\rtimes D_{2}$%
}%
\box\matricesbox
\else
\hbox to \myboxwidth{%
\unhbox\onelinebox
}%
\fi
\else
\hbox to \myboxwidth{%
$\Pi_{24,5}$ spans $L_{251.3}$%
\hfil}%
\hbox to \myboxwidth{%
$6222622622|226226222622|22\rtimes D_{2}$%
\hfil}%
\box\matricesbox
\fi
}%
\hfill\discretionary{}{}{}%
\setbox\matricesbox=\hbox{%
{$\left[\!\llap{\phantom{%
\begingroup \smaller\smaller\smaller
\endgroup%
}}\!\right]$}%
}%
\ifdim\wd\matricesbox>\halfwidth\myboxwidth=\hsize\else\myboxwidth=\halfwidth\fi
\vbox{%
\ifdim\myboxwidth=\hsize
\setbox\onelinebox=\hbox{%
\vbox{\hbox{%
$\Pi_{24,6}$ spans $L_{251.3}$%
}\hbox{%
$622262|262226222262|262222\rtimes D_{2}$%
}%
}%
\hfill\copy\matricesbox
}%
\ifdim\wd\onelinebox>\myboxwidth
\hbox to \myboxwidth{%
$\Pi_{24,6}$ spans $L_{251.3}$%
\hfil
$622262|262226222262|262222\rtimes D_{2}$%
}%
\box\matricesbox
\else
\hbox to \myboxwidth{%
\unhbox\onelinebox
}%
\fi
\else
\hbox to \myboxwidth{%
$\Pi_{24,6}$ spans $L_{251.3}$%
\hfil}%
\hbox to \myboxwidth{%
$622262|262226222262|262222\rtimes D_{2}$%
\hfil}%
\box\matricesbox
\fi
}%
\hfill\discretionary{}{}{}%
\setbox\matricesbox=\hbox{%
{$\left[\!\llap{\phantom{%
\begingroup \smaller\smaller\smaller
\endgroup%
}}\!\right]$}%
}%
\ifdim\wd\matricesbox>\halfwidth\myboxwidth=\hsize\else\myboxwidth=\halfwidth\fi
\vbox{%
\ifdim\myboxwidth=\hsize
\setbox\onelinebox=\hbox{%
\vbox{\hbox{%
$\Pi_{24,7}$ spans $L_{251.3}$%
}\hbox{%
$622262262222622262262222\rtimes C_{2}$%
}%
}%
\hfill\copy\matricesbox
}%
\ifdim\wd\onelinebox>\myboxwidth
\hbox to \myboxwidth{%
$\Pi_{24,7}$ spans $L_{251.3}$%
\hfil
$622262262222622262262222\rtimes C_{2}$%
}%
\box\matricesbox
\else
\hbox to \myboxwidth{%
\unhbox\onelinebox
}%
\fi
\else
\hbox to \myboxwidth{%
$\Pi_{24,7}$ spans $L_{251.3}$%
\hfil}%
\hbox to \myboxwidth{%
$622262262222622262262222\rtimes C_{2}$%
\hfil}%
\box\matricesbox
\fi
}%
\hfill\discretionary{}{}{}%
\setbox\matricesbox=\hbox{%
{$\left[\!\llap{\phantom{%
\begingroup \smaller\smaller\smaller
\endgroup%
}}\!\right]$}%
}%
\ifdim\wd\matricesbox>\halfwidth\myboxwidth=\hsize\else\myboxwidth=\halfwidth\fi
\vbox{%
\ifdim\myboxwidth=\hsize
\setbox\onelinebox=\hbox{%
\vbox{\hbox{%
$\Pi_{24,8}$ spans $L_{251.3}$%
}\hbox{%
$622262226222622262262222$%
}%
}%
\hfill\copy\matricesbox
}%
\ifdim\wd\onelinebox>\myboxwidth
\hbox to \myboxwidth{%
$\Pi_{24,8}$ spans $L_{251.3}$%
\hfil
$622262226222622262262222$%
}%
\box\matricesbox
\else
\hbox to \myboxwidth{%
\unhbox\onelinebox
}%
\fi
\else
\hbox to \myboxwidth{%
$\Pi_{24,8}$ spans $L_{251.3}$%
\hfil}%
\hbox to \myboxwidth{%
$622262226222622262262222$%
\hfil}%
\box\matricesbox
\fi
}%
\hfill\discretionary{}{}{}%
\setbox\matricesbox=\hbox{%
{$\left[\!\llap{\phantom{%
\begingroup \smaller\smaller\smaller
\endgroup%
}}\!\right]$}%
}%
\ifdim\wd\matricesbox>\halfwidth\myboxwidth=\hsize\else\myboxwidth=\halfwidth\fi
\vbox{%
\ifdim\myboxwidth=\hsize
\setbox\onelinebox=\hbox{%
\vbox{\hbox{%
$\Pi_{24,9}$ spans $L_{251.3}$%
}\hbox{%
$622262226222622622262222$%
}%
}%
\hfill\copy\matricesbox
}%
\ifdim\wd\onelinebox>\myboxwidth
\hbox to \myboxwidth{%
$\Pi_{24,9}$ spans $L_{251.3}$%
\hfil
$622262226222622622262222$%
}%
\box\matricesbox
\else
\hbox to \myboxwidth{%
\unhbox\onelinebox
}%
\fi
\else
\hbox to \myboxwidth{%
$\Pi_{24,9}$ spans $L_{251.3}$%
\hfil}%
\hbox to \myboxwidth{%
$622262226222622622262222$%
\hfil}%
\box\matricesbox
\fi
}%
\hfill\discretionary{}{}{}%
\setbox\matricesbox=\hbox{%
{$\left[\!\llap{\phantom{%
\begingroup \smaller\smaller\smaller
\endgroup%
}}\!\right]$}%
}%
\ifdim\wd\matricesbox>\halfwidth\myboxwidth=\hsize\else\myboxwidth=\halfwidth\fi
\vbox{%
\ifdim\myboxwidth=\hsize
\setbox\onelinebox=\hbox{%
\vbox{\hbox{%
$\Pi_{24,10}$ spans $L_{251.3}$%
}\hbox{%
$622262226226222262262222$%
}%
}%
\hfill\copy\matricesbox
}%
\ifdim\wd\onelinebox>\myboxwidth
\hbox to \myboxwidth{%
$\Pi_{24,10}$ spans $L_{251.3}$%
\hfil
$622262226226222262262222$%
}%
\box\matricesbox
\else
\hbox to \myboxwidth{%
\unhbox\onelinebox
}%
\fi
\else
\hbox to \myboxwidth{%
$\Pi_{24,10}$ spans $L_{251.3}$%
\hfil}%
\hbox to \myboxwidth{%
$622262226226222262262222$%
\hfil}%
\box\matricesbox
\fi
}%
\hfill\discretionary{}{}{}%

}{%
\leavevmode
\setbox\matricesbox=\hbox{%
{$\left[\!\llap{\phantom{%
\begingroup \smaller\smaller\smaller\begin{tabular}{@{}c@{}}%
\phantom{0}\\\phantom{0}\\\phantom{0}
\end{tabular}\endgroup%
}}\right.$}%
\begingroup \smaller\smaller\smaller\begin{tabular}{@{}c@{}}%
-1/4\\\phantom{0}\\\phantom{0}
\end{tabular}\endgroup%
\kern3pt%
\begingroup \smaller\smaller\smaller\begin{tabular}{@{}c@{}}%
\phantom{0}\\3\\\phantom{0}
\end{tabular}\endgroup%
\kern3pt%
\begingroup \smaller\smaller\smaller\begin{tabular}{@{}c@{}}%
\phantom{0}\\\phantom{0}\\3
\end{tabular}\endgroup%
{$\left.\llap{\phantom{%
\begingroup \smaller\smaller\smaller\begin{tabular}{@{}c@{}}%
\phantom{0}\\\phantom{0}\\\phantom{0}
\end{tabular}\endgroup%
}}\!\right]$}%
{$\left[\!\llap{\phantom{%
\begingroup \smaller\smaller\smaller\begin{tabular}{@{}c@{}}%
0\\0\\0
\end{tabular}\endgroup%
}}\right.$}%
\begingroup \smaller\smaller\smaller\begin{tabular}{@{}c@{}}%
2\\1\\0
\end{tabular}\endgroup%
{$\left.\llap{\phantom{%
\begingroup \smaller\smaller\smaller\begin{tabular}{@{}c@{}}%
0\\0\\0
\end{tabular}\endgroup%
}}\!\right]$}%
}%
\ifdim\wd\matricesbox>\halfwidth\myboxwidth=\hsize\else\myboxwidth=\halfwidth\fi
\vbox{%
\ifdim\myboxwidth=\hsize
\setbox\onelinebox=\hbox{%
\vbox{\hbox{%
$\Pi_{4,1}=B_2=\hbox{GN}_{29}$ spans $L_{3.4}$%
}\hbox{%
$|\slashthree|\slashthree|\slashthree|\slashthree\rtimes D_{8}$%
}%
}%
\hfill\copy\matricesbox
}%
\ifdim\wd\onelinebox>\myboxwidth
\hbox to \myboxwidth{%
$\Pi_{4,1}=B_2=\hbox{GN}_{29}$ spans $L_{3.4}$%
\hfil
$|\slashthree|\slashthree|\slashthree|\slashthree\rtimes D_{8}$%
}%
\box\matricesbox
\else
\hbox to \myboxwidth{%
\unhbox\onelinebox
}%
\fi
\else
\hbox to \myboxwidth{%
$\Pi_{4,1}=B_2=\hbox{GN}_{29}$ spans $L_{3.4}$%
\hfil}%
\hbox to \myboxwidth{%
$|\slashthree|\slashthree|\slashthree|\slashthree\rtimes D_{8}$%
\hfil}%
\box\matricesbox
\fi
}%
\hfill\discretionary{}{}{}%
\setbox\matricesbox=\hbox{%
{$\left[\!\llap{\phantom{%
\begingroup \smaller\smaller\smaller\begin{tabular}{@{}c@{}}%
\phantom{0}\\\phantom{0}\\\phantom{0}
\end{tabular}\endgroup%
}}\right.$}%
\begingroup \smaller\smaller\smaller\begin{tabular}{@{}c@{}}%
-1\\\phantom{0}\\\phantom{0}
\end{tabular}\endgroup%
\kern3pt%
\begingroup \smaller\smaller\smaller\begin{tabular}{@{}c@{}}%
\phantom{0}\\2\\\phantom{0}
\end{tabular}\endgroup%
\kern3pt%
\begingroup \smaller\smaller\smaller\begin{tabular}{@{}c@{}}%
\phantom{0}\\\phantom{0}\\2
\end{tabular}\endgroup%
{$\left.\llap{\phantom{%
\begingroup \smaller\smaller\smaller\begin{tabular}{@{}c@{}}%
\phantom{0}\\\phantom{0}\\\phantom{0}
\end{tabular}\endgroup%
}}\!\right]$}%
{$\left[\!\llap{\phantom{%
\begingroup \smaller\smaller\smaller\begin{tabular}{@{}c@{}}%
0\\0\\0
\end{tabular}\endgroup%
}}\right.$}%
\begingroup \smaller\smaller\smaller\begin{tabular}{@{}c@{}}%
1\\0\\-1
\end{tabular}\endgroup%
{$\left.\llap{\phantom{%
\begingroup \smaller\smaller\smaller\begin{tabular}{@{}c@{}}%
0\\0\\0
\end{tabular}\endgroup%
}}\!\right]$}%
}%
\ifdim\wd\matricesbox>\halfwidth\myboxwidth=\hsize\else\myboxwidth=\halfwidth\fi
\vbox{%
\ifdim\myboxwidth=\hsize
\setbox\onelinebox=\hbox{%
\vbox{\hbox{%
$\Pi_{4,2}=A_{2,II}=\hbox{GN}_{35}$ spans $L_{1.9}$%
}\hbox{%
$|\slashinfty|\slashinfty|\slashinfty|\slashinfty\rtimes D_{8}$ (shared)%
}%
}%
\hfill\copy\matricesbox
}%
\ifdim\wd\onelinebox>\myboxwidth
\hbox to \myboxwidth{%
$\Pi_{4,2}=A_{2,II}=\hbox{GN}_{35}$ spans $L_{1.9}$%
\hfil
$|\slashinfty|\slashinfty|\slashinfty|\slashinfty\rtimes D_{8}$ (shared)%
}%
\box\matricesbox
\else
\hbox to \myboxwidth{%
\unhbox\onelinebox
}%
\fi
\else
\hbox to \myboxwidth{%
$\Pi_{4,2}=A_{2,II}=\hbox{GN}_{35}$ spans $L_{1.9}$%
\hfil}%
\hbox to \myboxwidth{%
$|\slashinfty|\slashinfty|\slashinfty|\slashinfty\rtimes D_{8}$ (shared)%
\hfil}%
\box\matricesbox
\fi
}%
\hfill\discretionary{}{}{}%
\setbox\matricesbox=\hbox{%
{$\left[\!\llap{\phantom{%
\begingroup \smaller\smaller\smaller\begin{tabular}{@{}c@{}}%
\phantom{0}\\\phantom{0}\\\phantom{0}
\end{tabular}\endgroup%
}}\right.$}%
\begingroup \smaller\smaller\smaller\begin{tabular}{@{}c@{}}%
-3/8\\\phantom{0}\\\phantom{0}
\end{tabular}\endgroup%
\kern3pt%
\begingroup \smaller\smaller\smaller\begin{tabular}{@{}c@{}}%
\phantom{0}\\2\\\phantom{0}
\end{tabular}\endgroup%
\kern3pt%
\begingroup \smaller\smaller\smaller\begin{tabular}{@{}c@{}}%
\phantom{0}\\\phantom{0}\\3/2
\end{tabular}\endgroup%
{$\left.\llap{\phantom{%
\begingroup \smaller\smaller\smaller\begin{tabular}{@{}c@{}}%
\phantom{0}\\\phantom{0}\\\phantom{0}
\end{tabular}\endgroup%
}}\!\right]$}%
{$\left[\!\llap{\phantom{%
\begingroup \smaller\smaller\smaller\begin{tabular}{@{}c@{}}%
0\\0\\0
\end{tabular}\endgroup%
}}\right.$}%
\begingroup \smaller\smaller\smaller\begin{tabular}{@{}c@{}}%
2\\-1\\1
\end{tabular}\endgroup%
{$\left.\llap{\phantom{%
\begingroup \smaller\smaller\smaller\begin{tabular}{@{}c@{}}%
0\\0\\0
\end{tabular}\endgroup%
}}\!\right]$}%
}%
\ifdim\wd\matricesbox>\halfwidth\myboxwidth=\hsize\else\myboxwidth=\halfwidth\fi
\vbox{%
\ifdim\myboxwidth=\hsize
\setbox\onelinebox=\hbox{%
\vbox{\hbox{%
$\Pi_{4,3}=A_{3,I}=\hbox{GN}_{33}$ spans $L_{7.9}$%
}\hbox{%
$\slashthree\slashinfty\slashthree\slashinfty\rtimes D_{4}$%
}%
}%
\hfill\copy\matricesbox
}%
\ifdim\wd\onelinebox>\myboxwidth
\hbox to \myboxwidth{%
$\Pi_{4,3}=A_{3,I}=\hbox{GN}_{33}$ spans $L_{7.9}$%
\hfil
$\slashthree\slashinfty\slashthree\slashinfty\rtimes D_{4}$%
}%
\box\matricesbox
\else
\hbox to \myboxwidth{%
\unhbox\onelinebox
}%
\fi
\else
\hbox to \myboxwidth{%
$\Pi_{4,3}=A_{3,I}=\hbox{GN}_{33}$ spans $L_{7.9}$%
\hfil}%
\hbox to \myboxwidth{%
$\slashthree\slashinfty\slashthree\slashinfty\rtimes D_{4}$%
\hfil}%
\box\matricesbox
\fi
}%
\hfill\discretionary{}{}{}%
\setbox\matricesbox=\hbox{%
{$\left[\!\llap{\phantom{%
\begingroup \smaller\smaller\smaller\begin{tabular}{@{}c@{}}%
\phantom{0}\\\phantom{0}\\\phantom{0}
\end{tabular}\endgroup%
}}\right.$}%
\begingroup \smaller\smaller\smaller\begin{tabular}{@{}c@{}}%
-1/2\\\phantom{0}\\\phantom{0}
\end{tabular}\endgroup%
\kern3pt%
\begingroup \smaller\smaller\smaller\begin{tabular}{@{}c@{}}%
\phantom{0}\\3/2\\\phantom{0}
\end{tabular}\endgroup%
\kern3pt%
\begingroup \smaller\smaller\smaller\begin{tabular}{@{}c@{}}%
\phantom{0}\\\phantom{0}\\4
\end{tabular}\endgroup%
{$\left.\llap{\phantom{%
\begingroup \smaller\smaller\smaller\begin{tabular}{@{}c@{}}%
\phantom{0}\\\phantom{0}\\\phantom{0}
\end{tabular}\endgroup%
}}\!\right]$}%
{$\left[\!\llap{\phantom{%
\begingroup \smaller\smaller\smaller\begin{tabular}{@{}c@{}}%
0\\0\\0
\end{tabular}\endgroup%
}}\right.$}%
\begingroup \smaller\smaller\smaller\begin{tabular}{@{}c@{}}%
1\\-1\\0
\end{tabular}\endgroup%
\kern3pt%
\begingroup \smaller\smaller\smaller\begin{tabular}{@{}c@{}}%
2\\0\\-1
\end{tabular}\endgroup%
{$\left.\llap{\phantom{%
\begingroup \smaller\smaller\smaller\begin{tabular}{@{}c@{}}%
0\\0\\0
\end{tabular}\endgroup%
}}\!\right]$}%
}%
\ifdim\wd\matricesbox>\halfwidth\myboxwidth=\hsize\else\myboxwidth=\halfwidth\fi
\vbox{%
\ifdim\myboxwidth=\hsize
\setbox\onelinebox=\hbox{%
\vbox{\hbox{%
$\Pi_{4,4}$ spans $L_{123.4}$%
}\hbox{%
$|4|4|4|4\rtimes D_{4}$%
}%
}%
\hfill\copy\matricesbox
}%
\ifdim\wd\onelinebox>\myboxwidth
\hbox to \myboxwidth{%
$\Pi_{4,4}$ spans $L_{123.4}$%
\hfil
$|4|4|4|4\rtimes D_{4}$%
}%
\box\matricesbox
\else
\hbox to \myboxwidth{%
\unhbox\onelinebox
}%
\fi
\else
\hbox to \myboxwidth{%
$\Pi_{4,4}$ spans $L_{123.4}$%
\hfil}%
\hbox to \myboxwidth{%
$|4|4|4|4\rtimes D_{4}$%
\hfil}%
\box\matricesbox
\fi
}%
\hfill\discretionary{}{}{}%
\setbox\matricesbox=\hbox{%
{$\left[\!\llap{\phantom{%
\begingroup \smaller\smaller\smaller\begin{tabular}{@{}c@{}}%
\phantom{0}\\\phantom{0}\\\phantom{0}
\end{tabular}\endgroup%
}}\right.$}%
\begingroup \smaller\smaller\smaller\begin{tabular}{@{}c@{}}%
-1/4\\\phantom{0}\\\phantom{0}
\end{tabular}\endgroup%
\kern3pt%
\begingroup \smaller\smaller\smaller\begin{tabular}{@{}c@{}}%
\phantom{0}\\3\\\phantom{0}
\end{tabular}\endgroup%
\kern3pt%
\begingroup \smaller\smaller\smaller\begin{tabular}{@{}c@{}}%
\phantom{0}\\\phantom{0}\\15
\end{tabular}\endgroup%
{$\left.\llap{\phantom{%
\begingroup \smaller\smaller\smaller\begin{tabular}{@{}c@{}}%
\phantom{0}\\\phantom{0}\\\phantom{0}
\end{tabular}\endgroup%
}}\!\right]$}%
{$\left[\!\llap{\phantom{%
\begingroup \smaller\smaller\smaller\begin{tabular}{@{}c@{}}%
0\\0\\0
\end{tabular}\endgroup%
}}\right.$}%
\begingroup \smaller\smaller\smaller\begin{tabular}{@{}c@{}}%
2\\1\\0
\end{tabular}\endgroup%
\kern3pt%
\begingroup \smaller\smaller\smaller\begin{tabular}{@{}c@{}}%
6\\0\\1
\end{tabular}\endgroup%
{$\left.\llap{\phantom{%
\begingroup \smaller\smaller\smaller\begin{tabular}{@{}c@{}}%
0\\0\\0
\end{tabular}\endgroup%
}}\!\right]$}%
}%
\ifdim\wd\matricesbox>\halfwidth\myboxwidth=\hsize\else\myboxwidth=\halfwidth\fi
\vbox{%
\ifdim\myboxwidth=\hsize
\setbox\onelinebox=\hbox{%
\vbox{\hbox{%
$\Pi_{4,5}$ spans $L_{16.8}$%
}\hbox{%
$|6|6|6|6\rtimes D_{4}$%
}%
}%
\hfill\copy\matricesbox
}%
\ifdim\wd\onelinebox>\myboxwidth
\hbox to \myboxwidth{%
$\Pi_{4,5}$ spans $L_{16.8}$%
\hfil
$|6|6|6|6\rtimes D_{4}$%
}%
\box\matricesbox
\else
\hbox to \myboxwidth{%
\unhbox\onelinebox
}%
\fi
\else
\hbox to \myboxwidth{%
$\Pi_{4,5}$ spans $L_{16.8}$%
\hfil}%
\hbox to \myboxwidth{%
$|6|6|6|6\rtimes D_{4}$%
\hfil}%
\box\matricesbox
\fi
}%
\hfill\discretionary{}{}{}%
\setbox\matricesbox=\hbox{%
{$\left[\!\llap{\phantom{%
\begingroup \smaller\smaller\smaller\begin{tabular}{@{}c@{}}%
\phantom{0}\\\phantom{0}\\\phantom{0}
\end{tabular}\endgroup%
}}\right.$}%
\begingroup \smaller\smaller\smaller\begin{tabular}{@{}c@{}}%
-1/2\\\phantom{0}\\\phantom{0}
\end{tabular}\endgroup%
\kern3pt%
\begingroup \smaller\smaller\smaller\begin{tabular}{@{}c@{}}%
\phantom{0}\\3/2\\\phantom{0}
\end{tabular}\endgroup%
\kern3pt%
\begingroup \smaller\smaller\smaller\begin{tabular}{@{}c@{}}%
\phantom{0}\\\phantom{0}\\12
\end{tabular}\endgroup%
{$\left.\llap{\phantom{%
\begingroup \smaller\smaller\smaller\begin{tabular}{@{}c@{}}%
\phantom{0}\\\phantom{0}\\\phantom{0}
\end{tabular}\endgroup%
}}\!\right]$}%
{$\left[\!\llap{\phantom{%
\begingroup \smaller\smaller\smaller\begin{tabular}{@{}c@{}}%
0\\0\\0
\end{tabular}\endgroup%
}}\right.$}%
\begingroup \smaller\smaller\smaller\begin{tabular}{@{}c@{}}%
1\\-1\\0
\end{tabular}\endgroup%
\kern3pt%
\begingroup \smaller\smaller\smaller\begin{tabular}{@{}c@{}}%
4\\0\\-1
\end{tabular}\endgroup%
{$\left.\llap{\phantom{%
\begingroup \smaller\smaller\smaller\begin{tabular}{@{}c@{}}%
0\\0\\0
\end{tabular}\endgroup%
}}\!\right]$}%
}%
\ifdim\wd\matricesbox>\halfwidth\myboxwidth=\hsize\else\myboxwidth=\halfwidth\fi
\vbox{%
\ifdim\myboxwidth=\hsize
\setbox\onelinebox=\hbox{%
\vbox{\hbox{%
$\Pi_{4,6}=\hbox{GN}_{30}$ spans $L_{4.17}$%
}\hbox{%
$|\infty|\infty|\infty|\infty\rtimes D_{4}$%
}%
}%
\hfill\copy\matricesbox
}%
\ifdim\wd\onelinebox>\myboxwidth
\hbox to \myboxwidth{%
$\Pi_{4,6}=\hbox{GN}_{30}$ spans $L_{4.17}$%
\hfil
$|\infty|\infty|\infty|\infty\rtimes D_{4}$%
}%
\box\matricesbox
\else
\hbox to \myboxwidth{%
\unhbox\onelinebox
}%
\fi
\else
\hbox to \myboxwidth{%
$\Pi_{4,6}=\hbox{GN}_{30}$ spans $L_{4.17}$%
\hfil}%
\hbox to \myboxwidth{%
$|\infty|\infty|\infty|\infty\rtimes D_{4}$%
\hfil}%
\box\matricesbox
\fi
}%
\hfill\discretionary{}{}{}%
\setbox\matricesbox=\hbox{%
{$\left[\!\llap{\phantom{%
\begingroup \smaller\smaller\smaller\begin{tabular}{@{}c@{}}%
\phantom{0}\\\phantom{0}\\\phantom{0}
\end{tabular}\endgroup%
}}\right.$}%
\begingroup \smaller\smaller\smaller\begin{tabular}{@{}c@{}}%
-1/8\\\phantom{0}\\\phantom{0}
\end{tabular}\endgroup%
\kern3pt%
\begingroup \smaller\smaller\smaller\begin{tabular}{@{}c@{}}%
\phantom{0}\\1\\\phantom{0}
\end{tabular}\endgroup%
\kern3pt%
\begingroup \smaller\smaller\smaller\begin{tabular}{@{}c@{}}%
\phantom{0}\\\phantom{0}\\3/2
\end{tabular}\endgroup%
{$\left.\llap{\phantom{%
\begingroup \smaller\smaller\smaller\begin{tabular}{@{}c@{}}%
\phantom{0}\\\phantom{0}\\\phantom{0}
\end{tabular}\endgroup%
}}\!\right]$}%
{$\left[\!\llap{\phantom{%
\begingroup \smaller\smaller\smaller\begin{tabular}{@{}c@{}}%
0\\0\\0
\end{tabular}\endgroup%
}}\right.$}%
\begingroup \smaller\smaller\smaller\begin{tabular}{@{}c@{}}%
2\\-1\\-1
\end{tabular}\endgroup%
{$\left.\llap{\phantom{%
\begingroup \smaller\smaller\smaller\begin{tabular}{@{}c@{}}%
0\\0\\0
\end{tabular}\endgroup%
}}\!\right]$}%
}%
\ifdim\wd\matricesbox>\halfwidth\myboxwidth=\hsize\else\myboxwidth=\halfwidth\fi
\vbox{%
\ifdim\myboxwidth=\hsize
\setbox\onelinebox=\hbox{%
\vbox{\hbox{%
$\Pi_{4,7}=B_1=\hbox{GN}_{21}$ spans $L_{3.1}$%
}\hbox{%
$\slashthree\slashtwo\slashthree\slashtwo\rtimes D_{4}$ (shared)%
}%
}%
\hfill\copy\matricesbox
}%
\ifdim\wd\onelinebox>\myboxwidth
\hbox to \myboxwidth{%
$\Pi_{4,7}=B_1=\hbox{GN}_{21}$ spans $L_{3.1}$%
\hfil
$\slashthree\slashtwo\slashthree\slashtwo\rtimes D_{4}$ (shared)%
}%
\box\matricesbox
\else
\hbox to \myboxwidth{%
\unhbox\onelinebox
}%
\fi
\else
\hbox to \myboxwidth{%
$\Pi_{4,7}=B_1=\hbox{GN}_{21}$ spans $L_{3.1}$%
\hfil}%
\hbox to \myboxwidth{%
$\slashthree\slashtwo\slashthree\slashtwo\rtimes D_{4}$ (shared)%
\hfil}%
\box\matricesbox
\fi
}%
\hfill\discretionary{}{}{}%
\setbox\matricesbox=\hbox{%
{$\left[\!\llap{\phantom{%
\begingroup \smaller\smaller\smaller\begin{tabular}{@{}c@{}}%
\phantom{0}\\\phantom{0}\\\phantom{0}
\end{tabular}\endgroup%
}}\right.$}%
\begingroup \smaller\smaller\smaller\begin{tabular}{@{}c@{}}%
-1/2\\\phantom{0}\\\phantom{0}
\end{tabular}\endgroup%
\kern3pt%
\begingroup \smaller\smaller\smaller\begin{tabular}{@{}c@{}}%
\phantom{0}\\1/2\\\phantom{0}
\end{tabular}\endgroup%
\kern3pt%
\begingroup \smaller\smaller\smaller\begin{tabular}{@{}c@{}}%
\phantom{0}\\\phantom{0}\\1
\end{tabular}\endgroup%
{$\left.\llap{\phantom{%
\begingroup \smaller\smaller\smaller\begin{tabular}{@{}c@{}}%
\phantom{0}\\\phantom{0}\\\phantom{0}
\end{tabular}\endgroup%
}}\!\right]$}%
{$\left[\!\llap{\phantom{%
\begingroup \smaller\smaller\smaller\begin{tabular}{@{}c@{}}%
0\\0\\0
\end{tabular}\endgroup%
}}\right.$}%
\begingroup \smaller\smaller\smaller\begin{tabular}{@{}c@{}}%
1\\-1\\-1
\end{tabular}\endgroup%
{$\left.\llap{\phantom{%
\begingroup \smaller\smaller\smaller\begin{tabular}{@{}c@{}}%
0\\0\\0
\end{tabular}\endgroup%
}}\!\right]$}%
}%
\ifdim\wd\matricesbox>\halfwidth\myboxwidth=\hsize\else\myboxwidth=\halfwidth\fi
\vbox{%
\ifdim\myboxwidth=\hsize
\setbox\onelinebox=\hbox{%
\vbox{\hbox{%
$\Pi_{4,8}=A_{2,I}=\hbox{GN}_{27}$ spans $L_{1.6}$%
}\hbox{%
$\slashinfty\slashtwo\slashinfty\slashtwo\rtimes D_{4}$ (shared)%
}%
}%
\hfill\copy\matricesbox
}%
\ifdim\wd\onelinebox>\myboxwidth
\hbox to \myboxwidth{%
$\Pi_{4,8}=A_{2,I}=\hbox{GN}_{27}$ spans $L_{1.6}$%
\hfil
$\slashinfty\slashtwo\slashinfty\slashtwo\rtimes D_{4}$ (shared)%
}%
\box\matricesbox
\else
\hbox to \myboxwidth{%
\unhbox\onelinebox
}%
\fi
\else
\hbox to \myboxwidth{%
$\Pi_{4,8}=A_{2,I}=\hbox{GN}_{27}$ spans $L_{1.6}$%
\hfil}%
\hbox to \myboxwidth{%
$\slashinfty\slashtwo\slashinfty\slashtwo\rtimes D_{4}$ (shared)%
\hfil}%
\box\matricesbox
\fi
}%
\hfill\discretionary{}{}{}%
\setbox\matricesbox=\hbox{%
{$\left[\!\llap{\phantom{%
\begingroup \smaller\smaller\smaller\begin{tabular}{@{}c@{}}%
\phantom{0}\\\phantom{0}\\\phantom{0}
\end{tabular}\endgroup%
}}\right.$}%
\begingroup \smaller\smaller\smaller\begin{tabular}{@{}c@{}}%
-1/7\\\phantom{0}\\\phantom{0}
\end{tabular}\endgroup%
\kern3pt%
\begingroup \smaller\smaller\smaller\begin{tabular}{@{}c@{}}%
\phantom{0}\\1/14\\\phantom{0}
\end{tabular}\endgroup%
\kern3pt%
\begingroup \smaller\smaller\smaller\begin{tabular}{@{}c@{}}%
\phantom{0}\\\phantom{0}\\5/2
\end{tabular}\endgroup%
{$\left.\llap{\phantom{%
\begingroup \smaller\smaller\smaller\begin{tabular}{@{}c@{}}%
\phantom{0}\\\phantom{0}\\\phantom{0}
\end{tabular}\endgroup%
}}\!\right]$}%
{$\left[\!\llap{\phantom{%
\begingroup \smaller\smaller\smaller\begin{tabular}{@{}c@{}}%
0\\0\\0
\end{tabular}\endgroup%
}}\right.$}%
\begingroup \smaller\smaller\smaller\begin{tabular}{@{}c@{}}%
2\\-6\\0
\end{tabular}\endgroup%
\kern3pt%
\begingroup \smaller\smaller\smaller\begin{tabular}{@{}c@{}}%
2\\1\\1
\end{tabular}\endgroup%
\kern3pt%
\begingroup \smaller\smaller\smaller\begin{tabular}{@{}c@{}}%
1\\4\\0
\end{tabular}\endgroup%
{$\left.\llap{\phantom{%
\begingroup \smaller\smaller\smaller\begin{tabular}{@{}c@{}}%
0\\0\\0
\end{tabular}\endgroup%
}}\!\right]$}%
}%
\ifdim\wd\matricesbox>\halfwidth\myboxwidth=\hsize\else\myboxwidth=\halfwidth\fi
\vbox{%
\ifdim\myboxwidth=\hsize
\setbox\onelinebox=\hbox{%
\vbox{\hbox{%
$\Pi_{4,9}$ spans $L_{6.2}$%
}\hbox{%
$3|32|2\rtimes D_{2}$%
}%
}%
\hfill\copy\matricesbox
}%
\ifdim\wd\onelinebox>\myboxwidth
\hbox to \myboxwidth{%
$\Pi_{4,9}$ spans $L_{6.2}$%
\hfil
$3|32|2\rtimes D_{2}$%
}%
\box\matricesbox
\else
\hbox to \myboxwidth{%
\unhbox\onelinebox
}%
\fi
\else
\hbox to \myboxwidth{%
$\Pi_{4,9}$ spans $L_{6.2}$%
\hfil}%
\hbox to \myboxwidth{%
$3|32|2\rtimes D_{2}$%
\hfil}%
\box\matricesbox
\fi
}%
\hfill\discretionary{}{}{}%
\setbox\matricesbox=\hbox{%
{$\left[\!\llap{\phantom{%
\begingroup \smaller\smaller\smaller\begin{tabular}{@{}c@{}}%
\phantom{0}\\\phantom{0}\\\phantom{0}
\end{tabular}\endgroup%
}}\right.$}%
\begingroup \smaller\smaller\smaller\begin{tabular}{@{}c@{}}%
-6/25\\\phantom{0}\\\phantom{0}
\end{tabular}\endgroup%
\kern3pt%
\begingroup \smaller\smaller\smaller\begin{tabular}{@{}c@{}}%
\phantom{0}\\3/50\\\phantom{0}
\end{tabular}\endgroup%
\kern3pt%
\begingroup \smaller\smaller\smaller\begin{tabular}{@{}c@{}}%
\phantom{0}\\\phantom{0}\\1/2
\end{tabular}\endgroup%
{$\left.\llap{\phantom{%
\begingroup \smaller\smaller\smaller\begin{tabular}{@{}c@{}}%
\phantom{0}\\\phantom{0}\\\phantom{0}
\end{tabular}\endgroup%
}}\!\right]$}%
{$\left[\!\llap{\phantom{%
\begingroup \smaller\smaller\smaller\begin{tabular}{@{}c@{}}%
0\\0\\0
\end{tabular}\endgroup%
}}\right.$}%
\begingroup \smaller\smaller\smaller\begin{tabular}{@{}c@{}}%
6\\13\\3
\end{tabular}\endgroup%
\kern3pt%
\begingroup \smaller\smaller\smaller\begin{tabular}{@{}c@{}}%
2\\-4\\2
\end{tabular}\endgroup%
{$\left.\llap{\phantom{%
\begingroup \smaller\smaller\smaller\begin{tabular}{@{}c@{}}%
0\\0\\0
\end{tabular}\endgroup%
}}\!\right]$}%
}%
\ifdim\wd\matricesbox>\halfwidth\myboxwidth=\hsize\else\myboxwidth=\halfwidth\fi
\vbox{%
\ifdim\myboxwidth=\hsize
\setbox\onelinebox=\hbox{%
\vbox{\hbox{%
$\Pi_{4,10}$ spans $L_{7.9}$%
}\hbox{%
$\slashthree6\slashinfty6\rtimes D_{2}$ (shared)%
}%
}%
\hfill\copy\matricesbox
}%
\ifdim\wd\onelinebox>\myboxwidth
\hbox to \myboxwidth{%
$\Pi_{4,10}$ spans $L_{7.9}$%
\hfil
$\slashthree6\slashinfty6\rtimes D_{2}$ (shared)%
}%
\box\matricesbox
\else
\hbox to \myboxwidth{%
\unhbox\onelinebox
}%
\fi
\else
\hbox to \myboxwidth{%
$\Pi_{4,10}$ spans $L_{7.9}$%
\hfil}%
\hbox to \myboxwidth{%
$\slashthree6\slashinfty6\rtimes D_{2}$ (shared)%
\hfil}%
\box\matricesbox
\fi
}%
\hfill\discretionary{}{}{}%
\setbox\matricesbox=\hbox{%
{$\left[\!\llap{\phantom{%
\begingroup \smaller\smaller\smaller\begin{tabular}{@{}c@{}}%
\phantom{0}\\\phantom{0}\\\phantom{0}
\end{tabular}\endgroup%
}}\right.$}%
\begingroup \smaller\smaller\smaller\begin{tabular}{@{}c@{}}%
-2/9\\\phantom{0}\\\phantom{0}
\end{tabular}\endgroup%
\kern3pt%
\begingroup \smaller\smaller\smaller\begin{tabular}{@{}c@{}}%
\phantom{0}\\1/18\\\phantom{0}
\end{tabular}\endgroup%
\kern3pt%
\begingroup \smaller\smaller\smaller\begin{tabular}{@{}c@{}}%
\phantom{0}\\\phantom{0}\\3/2
\end{tabular}\endgroup%
{$\left.\llap{\phantom{%
\begingroup \smaller\smaller\smaller\begin{tabular}{@{}c@{}}%
\phantom{0}\\\phantom{0}\\\phantom{0}
\end{tabular}\endgroup%
}}\!\right]$}%
{$\left[\!\llap{\phantom{%
\begingroup \smaller\smaller\smaller\begin{tabular}{@{}c@{}}%
0\\0\\0
\end{tabular}\endgroup%
}}\right.$}%
\begingroup \smaller\smaller\smaller\begin{tabular}{@{}c@{}}%
2\\-5\\1
\end{tabular}\endgroup%
\kern3pt%
\begingroup \smaller\smaller\smaller\begin{tabular}{@{}c@{}}%
6\\12\\2
\end{tabular}\endgroup%
{$\left.\llap{\phantom{%
\begingroup \smaller\smaller\smaller\begin{tabular}{@{}c@{}}%
0\\0\\0
\end{tabular}\endgroup%
}}\!\right]$}%
}%
\ifdim\wd\matricesbox>\halfwidth\myboxwidth=\hsize\else\myboxwidth=\halfwidth\fi
\vbox{%
\ifdim\myboxwidth=\hsize
\setbox\onelinebox=\hbox{%
\vbox{\hbox{%
$\Pi_{4,11}$ spans $L_{7.13}$%
}\hbox{%
$\slashthree6\slashinfty6\rtimes D_{2}$ (shared)%
}%
}%
\hfill\copy\matricesbox
}%
\ifdim\wd\onelinebox>\myboxwidth
\hbox to \myboxwidth{%
$\Pi_{4,11}$ spans $L_{7.13}$%
\hfil
$\slashthree6\slashinfty6\rtimes D_{2}$ (shared)%
}%
\box\matricesbox
\else
\hbox to \myboxwidth{%
\unhbox\onelinebox
}%
\fi
\else
\hbox to \myboxwidth{%
$\Pi_{4,11}$ spans $L_{7.13}$%
\hfil}%
\hbox to \myboxwidth{%
$\slashthree6\slashinfty6\rtimes D_{2}$ (shared)%
\hfil}%
\box\matricesbox
\fi
}%
\hfill\discretionary{}{}{}%
\setbox\matricesbox=\hbox{%
{$\left[\!\llap{\phantom{%
\begingroup \smaller\smaller\smaller\begin{tabular}{@{}c@{}}%
\phantom{0}\\\phantom{0}\\\phantom{0}
\end{tabular}\endgroup%
}}\right.$}%
\begingroup \smaller\smaller\smaller\begin{tabular}{@{}c@{}}%
-1/20\\\phantom{0}\\\phantom{0}
\end{tabular}\endgroup%
\kern3pt%
\begingroup \smaller\smaller\smaller\begin{tabular}{@{}c@{}}%
\phantom{0}\\3/10\\\phantom{0}
\end{tabular}\endgroup%
\kern3pt%
\begingroup \smaller\smaller\smaller\begin{tabular}{@{}c@{}}%
\phantom{0}\\\phantom{0}\\15/2
\end{tabular}\endgroup%
{$\left.\llap{\phantom{%
\begingroup \smaller\smaller\smaller\begin{tabular}{@{}c@{}}%
\phantom{0}\\\phantom{0}\\\phantom{0}
\end{tabular}\endgroup%
}}\!\right]$}%
{$\left[\!\llap{\phantom{%
\begingroup \smaller\smaller\smaller\begin{tabular}{@{}c@{}}%
0\\0\\0
\end{tabular}\endgroup%
}}\right.$}%
\begingroup \smaller\smaller\smaller\begin{tabular}{@{}c@{}}%
12\\8\\0
\end{tabular}\endgroup%
\kern3pt%
\begingroup \smaller\smaller\smaller\begin{tabular}{@{}c@{}}%
6\\-1\\1
\end{tabular}\endgroup%
\kern3pt%
\begingroup \smaller\smaller\smaller\begin{tabular}{@{}c@{}}%
4\\-4\\0
\end{tabular}\endgroup%
{$\left.\llap{\phantom{%
\begingroup \smaller\smaller\smaller\begin{tabular}{@{}c@{}}%
0\\0\\0
\end{tabular}\endgroup%
}}\!\right]$}%
}%
\ifdim\wd\matricesbox>\halfwidth\myboxwidth=\hsize\else\myboxwidth=\halfwidth\fi
\vbox{%
\ifdim\myboxwidth=\hsize
\setbox\onelinebox=\hbox{%
\vbox{\hbox{%
$\Pi_{4,12}$ spans $L_{19.5}$%
}\hbox{%
$4|42|2\rtimes D_{2}$ (shared)%
}%
}%
\hfill\copy\matricesbox
}%
\ifdim\wd\onelinebox>\myboxwidth
\hbox to \myboxwidth{%
$\Pi_{4,12}$ spans $L_{19.5}$%
\hfil
$4|42|2\rtimes D_{2}$ (shared)%
}%
\box\matricesbox
\else
\hbox to \myboxwidth{%
\unhbox\onelinebox
}%
\fi
\else
\hbox to \myboxwidth{%
$\Pi_{4,12}$ spans $L_{19.5}$%
\hfil}%
\hbox to \myboxwidth{%
$4|42|2\rtimes D_{2}$ (shared)%
\hfil}%
\box\matricesbox
\fi
}%
\hfill\discretionary{}{}{}%
\setbox\matricesbox=\hbox{%
{$\left[\!\llap{\phantom{%
\begingroup \smaller\smaller\smaller\begin{tabular}{@{}c@{}}%
\phantom{0}\\\phantom{0}\\\phantom{0}
\end{tabular}\endgroup%
}}\right.$}%
\begingroup \smaller\smaller\smaller\begin{tabular}{@{}c@{}}%
-2/5\\\phantom{0}\\\phantom{0}
\end{tabular}\endgroup%
\kern3pt%
\begingroup \smaller\smaller\smaller\begin{tabular}{@{}c@{}}%
\phantom{0}\\1/10\\\phantom{0}
\end{tabular}\endgroup%
\kern3pt%
\begingroup \smaller\smaller\smaller\begin{tabular}{@{}c@{}}%
\phantom{0}\\\phantom{0}\\1/2
\end{tabular}\endgroup%
{$\left.\llap{\phantom{%
\begingroup \smaller\smaller\smaller\begin{tabular}{@{}c@{}}%
\phantom{0}\\\phantom{0}\\\phantom{0}
\end{tabular}\endgroup%
}}\!\right]$}%
{$\left[\!\llap{\phantom{%
\begingroup \smaller\smaller\smaller\begin{tabular}{@{}c@{}}%
0\\0\\0
\end{tabular}\endgroup%
}}\right.$}%
\begingroup \smaller\smaller\smaller\begin{tabular}{@{}c@{}}%
2\\4\\-2
\end{tabular}\endgroup%
\kern3pt%
\begingroup \smaller\smaller\smaller\begin{tabular}{@{}c@{}}%
1\\-3\\-1
\end{tabular}\endgroup%
{$\left.\llap{\phantom{%
\begingroup \smaller\smaller\smaller\begin{tabular}{@{}c@{}}%
0\\0\\0
\end{tabular}\endgroup%
}}\!\right]$}%
}%
\ifdim\wd\matricesbox>\halfwidth\myboxwidth=\hsize\else\myboxwidth=\halfwidth\fi
\vbox{%
\ifdim\myboxwidth=\hsize
\setbox\onelinebox=\hbox{%
\vbox{\hbox{%
$\Pi_{4,13}$ spans $L_{1.3}$%
}\hbox{%
$4\slashinfty4\slashtwo\rtimes D_{2}$ (shared)%
}%
}%
\hfill\copy\matricesbox
}%
\ifdim\wd\onelinebox>\myboxwidth
\hbox to \myboxwidth{%
$\Pi_{4,13}$ spans $L_{1.3}$%
\hfil
$4\slashinfty4\slashtwo\rtimes D_{2}$ (shared)%
}%
\box\matricesbox
\else
\hbox to \myboxwidth{%
\unhbox\onelinebox
}%
\fi
\else
\hbox to \myboxwidth{%
$\Pi_{4,13}$ spans $L_{1.3}$%
\hfil}%
\hbox to \myboxwidth{%
$4\slashinfty4\slashtwo\rtimes D_{2}$ (shared)%
\hfil}%
\box\matricesbox
\fi
}%
\hfill\discretionary{}{}{}%
\setbox\matricesbox=\hbox{%
{$\left[\!\llap{\phantom{%
\begingroup \smaller\smaller\smaller\begin{tabular}{@{}c@{}}%
\phantom{0}\\\phantom{0}\\\phantom{0}
\end{tabular}\endgroup%
}}\right.$}%
\begingroup \smaller\smaller\smaller\begin{tabular}{@{}c@{}}%
-1/13\\\phantom{0}\\\phantom{0}
\end{tabular}\endgroup%
\kern3pt%
\begingroup \smaller\smaller\smaller\begin{tabular}{@{}c@{}}%
\phantom{0}\\3/13\\\phantom{0}
\end{tabular}\endgroup%
\kern3pt%
\begingroup \smaller\smaller\smaller\begin{tabular}{@{}c@{}}%
\phantom{0}\\\phantom{0}\\5
\end{tabular}\endgroup%
{$\left.\llap{\phantom{%
\begingroup \smaller\smaller\smaller\begin{tabular}{@{}c@{}}%
\phantom{0}\\\phantom{0}\\\phantom{0}
\end{tabular}\endgroup%
}}\!\right]$}%
{$\left[\!\llap{\phantom{%
\begingroup \smaller\smaller\smaller\begin{tabular}{@{}c@{}}%
0\\0\\0
\end{tabular}\endgroup%
}}\right.$}%
\begingroup \smaller\smaller\smaller\begin{tabular}{@{}c@{}}%
12\\10\\0
\end{tabular}\endgroup%
\kern3pt%
\begingroup \smaller\smaller\smaller\begin{tabular}{@{}c@{}}%
4\\-1\\1
\end{tabular}\endgroup%
\kern3pt%
\begingroup \smaller\smaller\smaller\begin{tabular}{@{}c@{}}%
3\\-4\\0
\end{tabular}\endgroup%
{$\left.\llap{\phantom{%
\begingroup \smaller\smaller\smaller\begin{tabular}{@{}c@{}}%
0\\0\\0
\end{tabular}\endgroup%
}}\!\right]$}%
}%
\ifdim\wd\matricesbox>\halfwidth\myboxwidth=\hsize\else\myboxwidth=\halfwidth\fi
\vbox{%
\ifdim\myboxwidth=\hsize
\setbox\onelinebox=\hbox{%
\vbox{\hbox{%
$\Pi_{4,14}$ spans $L_{31.3}$%
}\hbox{%
$6|62|2\rtimes D_{2}$%
}%
}%
\hfill\copy\matricesbox
}%
\ifdim\wd\onelinebox>\myboxwidth
\hbox to \myboxwidth{%
$\Pi_{4,14}$ spans $L_{31.3}$%
\hfil
$6|62|2\rtimes D_{2}$%
}%
\box\matricesbox
\else
\hbox to \myboxwidth{%
\unhbox\onelinebox
}%
\fi
\else
\hbox to \myboxwidth{%
$\Pi_{4,14}$ spans $L_{31.3}$%
\hfil}%
\hbox to \myboxwidth{%
$6|62|2\rtimes D_{2}$%
\hfil}%
\box\matricesbox
\fi
}%
\hfill\discretionary{}{}{}%
\setbox\matricesbox=\hbox{%
{$\left[\!\llap{\phantom{%
\begingroup \smaller\smaller\smaller\begin{tabular}{@{}c@{}}%
\phantom{0}\\\phantom{0}\\\phantom{0}
\end{tabular}\endgroup%
}}\right.$}%
\begingroup \smaller\smaller\smaller\begin{tabular}{@{}c@{}}%
-1/7\\\phantom{0}\\\phantom{0}
\end{tabular}\endgroup%
\kern3pt%
\begingroup \smaller\smaller\smaller\begin{tabular}{@{}c@{}}%
\phantom{0}\\9/14\\\phantom{0}
\end{tabular}\endgroup%
\kern3pt%
\begingroup \smaller\smaller\smaller\begin{tabular}{@{}c@{}}%
\phantom{0}\\\phantom{0}\\21/2
\end{tabular}\endgroup%
{$\left.\llap{\phantom{%
\begingroup \smaller\smaller\smaller\begin{tabular}{@{}c@{}}%
\phantom{0}\\\phantom{0}\\\phantom{0}
\end{tabular}\endgroup%
}}\!\right]$}%
{$\left[\!\llap{\phantom{%
\begingroup \smaller\smaller\smaller\begin{tabular}{@{}c@{}}%
0\\0\\0
\end{tabular}\endgroup%
}}\right.$}%
\begingroup \smaller\smaller\smaller\begin{tabular}{@{}c@{}}%
18\\10\\0
\end{tabular}\endgroup%
\kern3pt%
\begingroup \smaller\smaller\smaller\begin{tabular}{@{}c@{}}%
6\\1\\-1
\end{tabular}\endgroup%
\kern3pt%
\begingroup \smaller\smaller\smaller\begin{tabular}{@{}c@{}}%
2\\-2\\0
\end{tabular}\endgroup%
{$\left.\llap{\phantom{%
\begingroup \smaller\smaller\smaller\begin{tabular}{@{}c@{}}%
0\\0\\0
\end{tabular}\endgroup%
}}\!\right]$}%
}%
\ifdim\wd\matricesbox>\halfwidth\myboxwidth=\hsize\else\myboxwidth=\halfwidth\fi
\vbox{%
\ifdim\myboxwidth=\hsize
\setbox\onelinebox=\hbox{%
\vbox{\hbox{%
$\Pi_{4,15}$ spans $L_{228.1}$%
}\hbox{%
$6|66|6\rtimes D_{2}$%
}%
}%
\hfill\copy\matricesbox
}%
\ifdim\wd\onelinebox>\myboxwidth
\hbox to \myboxwidth{%
$\Pi_{4,15}$ spans $L_{228.1}$%
\hfil
$6|66|6\rtimes D_{2}$%
}%
\box\matricesbox
\else
\hbox to \myboxwidth{%
\unhbox\onelinebox
}%
\fi
\else
\hbox to \myboxwidth{%
$\Pi_{4,15}$ spans $L_{228.1}$%
\hfil}%
\hbox to \myboxwidth{%
$6|66|6\rtimes D_{2}$%
\hfil}%
\box\matricesbox
\fi
}%
\hfill\discretionary{}{}{}%
\setbox\matricesbox=\hbox{%
{$\left[\!\llap{\phantom{%
\begingroup \smaller\smaller\smaller\begin{tabular}{@{}c@{}}%
\phantom{0}\\\phantom{0}\\\phantom{0}
\end{tabular}\endgroup%
}}\right.$}%
\begingroup \smaller\smaller\smaller\begin{tabular}{@{}c@{}}%
-1/32\\\phantom{0}\\\phantom{0}
\end{tabular}\endgroup%
\kern3pt%
\begingroup \smaller\smaller\smaller\begin{tabular}{@{}c@{}}%
\phantom{0}\\5/8\\\phantom{0}
\end{tabular}\endgroup%
\kern3pt%
\begingroup \smaller\smaller\smaller\begin{tabular}{@{}c@{}}%
\phantom{0}\\\phantom{0}\\25/2
\end{tabular}\endgroup%
{$\left.\llap{\phantom{%
\begingroup \smaller\smaller\smaller\begin{tabular}{@{}c@{}}%
\phantom{0}\\\phantom{0}\\\phantom{0}
\end{tabular}\endgroup%
}}\!\right]$}%
{$\left[\!\llap{\phantom{%
\begingroup \smaller\smaller\smaller\begin{tabular}{@{}c@{}}%
0\\0\\0
\end{tabular}\endgroup%
}}\right.$}%
\begingroup \smaller\smaller\smaller\begin{tabular}{@{}c@{}}%
40\\-12\\0
\end{tabular}\endgroup%
\kern3pt%
\begingroup \smaller\smaller\smaller\begin{tabular}{@{}c@{}}%
10\\1\\-1
\end{tabular}\endgroup%
\kern3pt%
\begingroup \smaller\smaller\smaller\begin{tabular}{@{}c@{}}%
8\\4\\0
\end{tabular}\endgroup%
{$\left.\llap{\phantom{%
\begingroup \smaller\smaller\smaller\begin{tabular}{@{}c@{}}%
0\\0\\0
\end{tabular}\endgroup%
}}\!\right]$}%
}%
\ifdim\wd\matricesbox>\halfwidth\myboxwidth=\hsize\else\myboxwidth=\halfwidth\fi
\vbox{%
\ifdim\myboxwidth=\hsize
\setbox\onelinebox=\hbox{%
\vbox{\hbox{%
$\Pi_{4,16}$ spans $L_{10.1}$%
}\hbox{%
$\infty|\infty2|2\rtimes D_{2}$%
}%
}%
\hfill\copy\matricesbox
}%
\ifdim\wd\onelinebox>\myboxwidth
\hbox to \myboxwidth{%
$\Pi_{4,16}$ spans $L_{10.1}$%
\hfil
$\infty|\infty2|2\rtimes D_{2}$%
}%
\box\matricesbox
\else
\hbox to \myboxwidth{%
\unhbox\onelinebox
}%
\fi
\else
\hbox to \myboxwidth{%
$\Pi_{4,16}$ spans $L_{10.1}$%
\hfil}%
\hbox to \myboxwidth{%
$\infty|\infty2|2\rtimes D_{2}$%
\hfil}%
\box\matricesbox
\fi
}%
\hfill\discretionary{}{}{}%
\setbox\matricesbox=\hbox{%
{$\left[\!\llap{\phantom{%
\begingroup \smaller\smaller\smaller\begin{tabular}{@{}c@{}}%
\phantom{0}\\\phantom{0}\\\phantom{0}
\end{tabular}\endgroup%
}}\right.$}%
\begingroup \smaller\smaller\smaller\begin{tabular}{@{}c@{}}%
-1/2\\\phantom{0}\\\phantom{0}
\end{tabular}\endgroup%
\kern3pt%
\begingroup \smaller\smaller\smaller\begin{tabular}{@{}c@{}}%
\phantom{0}\\1/2\\\phantom{0}
\end{tabular}\endgroup%
\kern3pt%
\begingroup \smaller\smaller\smaller\begin{tabular}{@{}c@{}}%
\phantom{0}\\\phantom{0}\\1
\end{tabular}\endgroup%
{$\left.\llap{\phantom{%
\begingroup \smaller\smaller\smaller\begin{tabular}{@{}c@{}}%
\phantom{0}\\\phantom{0}\\\phantom{0}
\end{tabular}\endgroup%
}}\!\right]$}%
{$\left[\!\llap{\phantom{%
\begingroup \smaller\smaller\smaller\begin{tabular}{@{}c@{}}%
0\\0\\0
\end{tabular}\endgroup%
}}\right.$}%
\begingroup \smaller\smaller\smaller\begin{tabular}{@{}c@{}}%
4\\4\\2
\end{tabular}\endgroup%
\kern3pt%
\begingroup \smaller\smaller\smaller\begin{tabular}{@{}c@{}}%
1\\-1\\1
\end{tabular}\endgroup%
{$\left.\llap{\phantom{%
\begingroup \smaller\smaller\smaller\begin{tabular}{@{}c@{}}%
0\\0\\0
\end{tabular}\endgroup%
}}\!\right]$}%
}%
\ifdim\wd\matricesbox>\halfwidth\myboxwidth=\hsize\else\myboxwidth=\halfwidth\fi
\vbox{%
\ifdim\myboxwidth=\hsize
\setbox\onelinebox=\hbox{%
\vbox{\hbox{%
$\Pi_{4,17}=\hbox{GN}_{31}$ spans $L_{140.4}$%
}\hbox{%
$\infty\slashinfty\infty\slashinfty\rtimes D_{2}$ (shared)%
}%
}%
\hfill\copy\matricesbox
}%
\ifdim\wd\onelinebox>\myboxwidth
\hbox to \myboxwidth{%
$\Pi_{4,17}=\hbox{GN}_{31}$ spans $L_{140.4}$%
\hfil
$\infty\slashinfty\infty\slashinfty\rtimes D_{2}$ (shared)%
}%
\box\matricesbox
\else
\hbox to \myboxwidth{%
\unhbox\onelinebox
}%
\fi
\else
\hbox to \myboxwidth{%
$\Pi_{4,17}=\hbox{GN}_{31}$ spans $L_{140.4}$%
\hfil}%
\hbox to \myboxwidth{%
$\infty\slashinfty\infty\slashinfty\rtimes D_{2}$ (shared)%
\hfil}%
\box\matricesbox
\fi
}%
\hfill\discretionary{}{}{}%
\setbox\matricesbox=\hbox{%
{$\left[\!\llap{\phantom{%
\begingroup \smaller\smaller\smaller\begin{tabular}{@{}c@{}}%
\phantom{0}\\\phantom{0}\\\phantom{0}
\end{tabular}\endgroup%
}}\right.$}%
\begingroup \smaller\smaller\smaller\begin{tabular}{@{}c@{}}%
-1/5\\\phantom{0}\\\phantom{0}
\end{tabular}\endgroup%
\kern3pt%
\begingroup \smaller\smaller\smaller\begin{tabular}{@{}c@{}}%
\phantom{0}\\3/10\\\phantom{0}
\end{tabular}\endgroup%
\kern3pt%
\begingroup \smaller\smaller\smaller\begin{tabular}{@{}c@{}}%
\phantom{0}\\\phantom{0}\\5/2
\end{tabular}\endgroup%
{$\left.\llap{\phantom{%
\begingroup \smaller\smaller\smaller\begin{tabular}{@{}c@{}}%
\phantom{0}\\\phantom{0}\\\phantom{0}
\end{tabular}\endgroup%
}}\!\right]$}%
{$\left[\!\llap{\phantom{%
\begingroup \smaller\smaller\smaller\begin{tabular}{@{}c@{}}%
0\\0\\0
\end{tabular}\endgroup%
}}\right.$}%
\begingroup \smaller\smaller\smaller\begin{tabular}{@{}c@{}}%
1\\-2\\0
\end{tabular}\endgroup%
\kern3pt%
\begingroup \smaller\smaller\smaller\begin{tabular}{@{}c@{}}%
2\\1\\-1
\end{tabular}\endgroup%
\kern3pt%
\begingroup \smaller\smaller\smaller\begin{tabular}{@{}c@{}}%
3\\4\\0
\end{tabular}\endgroup%
{$\left.\llap{\phantom{%
\begingroup \smaller\smaller\smaller\begin{tabular}{@{}c@{}}%
0\\0\\0
\end{tabular}\endgroup%
}}\!\right]$}%
}%
\ifdim\wd\matricesbox>\halfwidth\myboxwidth=\hsize\else\myboxwidth=\halfwidth\fi
\vbox{%
\ifdim\myboxwidth=\hsize
\setbox\onelinebox=\hbox{%
\vbox{\hbox{%
$\Pi_{4,18}$ spans $L_{19.2}$%
}\hbox{%
$4|42|2\rtimes D_{2}$ (shared)%
}%
}%
\hfill\copy\matricesbox
}%
\ifdim\wd\onelinebox>\myboxwidth
\hbox to \myboxwidth{%
$\Pi_{4,18}$ spans $L_{19.2}$%
\hfil
$4|42|2\rtimes D_{2}$ (shared)%
}%
\box\matricesbox
\else
\hbox to \myboxwidth{%
\unhbox\onelinebox
}%
\fi
\else
\hbox to \myboxwidth{%
$\Pi_{4,18}$ spans $L_{19.2}$%
\hfil}%
\hbox to \myboxwidth{%
$4|42|2\rtimes D_{2}$ (shared)%
\hfil}%
\box\matricesbox
\fi
}%
\hfill\discretionary{}{}{}%
\setbox\matricesbox=\hbox{%
{$\left[\!\llap{\phantom{%
\begingroup \smaller\smaller\smaller\begin{tabular}{@{}c@{}}%
\phantom{0}\\\phantom{0}\\\phantom{0}
\end{tabular}\endgroup%
}}\right.$}%
\begingroup \smaller\smaller\smaller\begin{tabular}{@{}c@{}}%
-1/7\\\phantom{0}\\\phantom{0}
\end{tabular}\endgroup%
\kern3pt%
\begingroup \smaller\smaller\smaller\begin{tabular}{@{}c@{}}%
\phantom{0}\\2/7\\\phantom{0}
\end{tabular}\endgroup%
\kern3pt%
\begingroup \smaller\smaller\smaller\begin{tabular}{@{}c@{}}%
\phantom{0}\\\phantom{0}\\6
\end{tabular}\endgroup%
{$\left.\llap{\phantom{%
\begingroup \smaller\smaller\smaller\begin{tabular}{@{}c@{}}%
\phantom{0}\\\phantom{0}\\\phantom{0}
\end{tabular}\endgroup%
}}\!\right]$}%
{$\left[\!\llap{\phantom{%
\begingroup \smaller\smaller\smaller\begin{tabular}{@{}c@{}}%
0\\0\\0
\end{tabular}\endgroup%
}}\right.$}%
\begingroup \smaller\smaller\smaller\begin{tabular}{@{}c@{}}%
2\\-3\\0
\end{tabular}\endgroup%
\kern3pt%
\begingroup \smaller\smaller\smaller\begin{tabular}{@{}c@{}}%
4\\1\\1
\end{tabular}\endgroup%
\kern3pt%
\begingroup \smaller\smaller\smaller\begin{tabular}{@{}c@{}}%
1\\2\\0
\end{tabular}\endgroup%
{$\left.\llap{\phantom{%
\begingroup \smaller\smaller\smaller\begin{tabular}{@{}c@{}}%
0\\0\\0
\end{tabular}\endgroup%
}}\!\right]$}%
}%
\ifdim\wd\matricesbox>\halfwidth\myboxwidth=\hsize\else\myboxwidth=\halfwidth\fi
\vbox{%
\ifdim\myboxwidth=\hsize
\setbox\onelinebox=\hbox{%
\vbox{\hbox{%
$\Pi_{4,19}$ spans $L_{123.5}$%
}\hbox{%
$4|42|2\rtimes D_{2}$ (shared)%
}%
}%
\hfill\copy\matricesbox
}%
\ifdim\wd\onelinebox>\myboxwidth
\hbox to \myboxwidth{%
$\Pi_{4,19}$ spans $L_{123.5}$%
\hfil
$4|42|2\rtimes D_{2}$ (shared)%
}%
\box\matricesbox
\else
\hbox to \myboxwidth{%
\unhbox\onelinebox
}%
\fi
\else
\hbox to \myboxwidth{%
$\Pi_{4,19}$ spans $L_{123.5}$%
\hfil}%
\hbox to \myboxwidth{%
$4|42|2\rtimes D_{2}$ (shared)%
\hfil}%
\box\matricesbox
\fi
}%
\hfill\discretionary{}{}{}%
\setbox\matricesbox=\hbox{%
{$\left[\!\llap{\phantom{%
\begingroup \smaller\smaller\smaller\begin{tabular}{@{}c@{}}%
\phantom{0}\\\phantom{0}\\\phantom{0}
\end{tabular}\endgroup%
}}\right.$}%
\begingroup \smaller\smaller\smaller\begin{tabular}{@{}c@{}}%
-4/9\\\phantom{0}\\\phantom{0}
\end{tabular}\endgroup%
\kern3pt%
\begingroup \smaller\smaller\smaller\begin{tabular}{@{}c@{}}%
\phantom{0}\\1/9\\\phantom{0}
\end{tabular}\endgroup%
\kern3pt%
\begingroup \smaller\smaller\smaller\begin{tabular}{@{}c@{}}%
\phantom{0}\\\phantom{0}\\1
\end{tabular}\endgroup%
{$\left.\llap{\phantom{%
\begingroup \smaller\smaller\smaller\begin{tabular}{@{}c@{}}%
\phantom{0}\\\phantom{0}\\\phantom{0}
\end{tabular}\endgroup%
}}\!\right]$}%
{$\left[\!\llap{\phantom{%
\begingroup \smaller\smaller\smaller\begin{tabular}{@{}c@{}}%
0\\0\\0
\end{tabular}\endgroup%
}}\right.$}%
\begingroup \smaller\smaller\smaller\begin{tabular}{@{}c@{}}%
1\\2\\1
\end{tabular}\endgroup%
\kern3pt%
\begingroup \smaller\smaller\smaller\begin{tabular}{@{}c@{}}%
2\\-5\\1
\end{tabular}\endgroup%
{$\left.\llap{\phantom{%
\begingroup \smaller\smaller\smaller\begin{tabular}{@{}c@{}}%
0\\0\\0
\end{tabular}\endgroup%
}}\!\right]$}%
}%
\ifdim\wd\matricesbox>\halfwidth\myboxwidth=\hsize\else\myboxwidth=\halfwidth\fi
\vbox{%
\ifdim\myboxwidth=\hsize
\setbox\onelinebox=\hbox{%
\vbox{\hbox{%
$\Pi_{4,20}$ spans $L_{145.1}$%
}\hbox{%
$4\slashinfty4\slashtwo\rtimes D_{2}$ (shared)%
}%
}%
\hfill\copy\matricesbox
}%
\ifdim\wd\onelinebox>\myboxwidth
\hbox to \myboxwidth{%
$\Pi_{4,20}$ spans $L_{145.1}$%
\hfil
$4\slashinfty4\slashtwo\rtimes D_{2}$ (shared)%
}%
\box\matricesbox
\else
\hbox to \myboxwidth{%
\unhbox\onelinebox
}%
\fi
\else
\hbox to \myboxwidth{%
$\Pi_{4,20}$ spans $L_{145.1}$%
\hfil}%
\hbox to \myboxwidth{%
$4\slashinfty4\slashtwo\rtimes D_{2}$ (shared)%
\hfil}%
\box\matricesbox
\fi
}%
\hfill\discretionary{}{}{}%
\setbox\matricesbox=\hbox{%
{$\left[\!\llap{\phantom{%
\begingroup \smaller\smaller\smaller\begin{tabular}{@{}c@{}}%
\phantom{0}\\\phantom{0}\\\phantom{0}
\end{tabular}\endgroup%
}}\right.$}%
\begingroup \smaller\smaller\smaller\begin{tabular}{@{}c@{}}%
-1/5\\\phantom{0}\\\phantom{0}
\end{tabular}\endgroup%
\kern3pt%
\begingroup \smaller\smaller\smaller\begin{tabular}{@{}c@{}}%
\phantom{0}\\3/10\\\phantom{0}
\end{tabular}\endgroup%
\kern3pt%
\begingroup \smaller\smaller\smaller\begin{tabular}{@{}c@{}}%
\phantom{0}\\\phantom{0}\\9/2
\end{tabular}\endgroup%
{$\left.\llap{\phantom{%
\begingroup \smaller\smaller\smaller\begin{tabular}{@{}c@{}}%
\phantom{0}\\\phantom{0}\\\phantom{0}
\end{tabular}\endgroup%
}}\!\right]$}%
{$\left[\!\llap{\phantom{%
\begingroup \smaller\smaller\smaller\begin{tabular}{@{}c@{}}%
0\\0\\0
\end{tabular}\endgroup%
}}\right.$}%
\begingroup \smaller\smaller\smaller\begin{tabular}{@{}c@{}}%
3\\4\\0
\end{tabular}\endgroup%
\kern3pt%
\begingroup \smaller\smaller\smaller\begin{tabular}{@{}c@{}}%
3\\-1\\1
\end{tabular}\endgroup%
\kern3pt%
\begingroup \smaller\smaller\smaller\begin{tabular}{@{}c@{}}%
1\\-2\\0
\end{tabular}\endgroup%
{$\left.\llap{\phantom{%
\begingroup \smaller\smaller\smaller\begin{tabular}{@{}c@{}}%
0\\0\\0
\end{tabular}\endgroup%
}}\!\right]$}%
}%
\ifdim\wd\matricesbox>\halfwidth\myboxwidth=\hsize\else\myboxwidth=\halfwidth\fi
\vbox{%
\ifdim\myboxwidth=\hsize
\setbox\onelinebox=\hbox{%
\vbox{\hbox{%
$\Pi_{4,21}$ spans $L_{4.10}$%
}\hbox{%
$\infty|\infty2|2\rtimes D_{2}$ (shared)%
}%
}%
\hfill\copy\matricesbox
}%
\ifdim\wd\onelinebox>\myboxwidth
\hbox to \myboxwidth{%
$\Pi_{4,21}$ spans $L_{4.10}$%
\hfil
$\infty|\infty2|2\rtimes D_{2}$ (shared)%
}%
\box\matricesbox
\else
\hbox to \myboxwidth{%
\unhbox\onelinebox
}%
\fi
\else
\hbox to \myboxwidth{%
$\Pi_{4,21}$ spans $L_{4.10}$%
\hfil}%
\hbox to \myboxwidth{%
$\infty|\infty2|2\rtimes D_{2}$ (shared)%
\hfil}%
\box\matricesbox
\fi
}%
\hfill\discretionary{}{}{}%
\setbox\matricesbox=\hbox{%
{$\left[\!\llap{\phantom{%
\begingroup \smaller\smaller\smaller\begin{tabular}{@{}c@{}}%
\phantom{0}\\\phantom{0}\\\phantom{0}
\end{tabular}\endgroup%
}}\right.$}%
\begingroup \smaller\smaller\smaller\begin{tabular}{@{}c@{}}%
-1/7\\\phantom{0}\\\phantom{0}
\end{tabular}\endgroup%
\kern3pt%
\begingroup \smaller\smaller\smaller\begin{tabular}{@{}c@{}}%
\phantom{0}\\1/14\\\phantom{0}
\end{tabular}\endgroup%
\kern3pt%
\begingroup \smaller\smaller\smaller\begin{tabular}{@{}c@{}}%
\phantom{0}\\\phantom{0}\\21/2
\end{tabular}\endgroup%
{$\left.\llap{\phantom{%
\begingroup \smaller\smaller\smaller\begin{tabular}{@{}c@{}}%
\phantom{0}\\\phantom{0}\\\phantom{0}
\end{tabular}\endgroup%
}}\!\right]$}%
{$\left[\!\llap{\phantom{%
\begingroup \smaller\smaller\smaller\begin{tabular}{@{}c@{}}%
0\\0\\0
\end{tabular}\endgroup%
}}\right.$}%
\begingroup \smaller\smaller\smaller\begin{tabular}{@{}c@{}}%
2\\6\\0
\end{tabular}\endgroup%
\kern3pt%
\begingroup \smaller\smaller\smaller\begin{tabular}{@{}c@{}}%
6\\-3\\1
\end{tabular}\endgroup%
\kern3pt%
\begingroup \smaller\smaller\smaller\begin{tabular}{@{}c@{}}%
1\\-4\\0
\end{tabular}\endgroup%
{$\left.\llap{\phantom{%
\begingroup \smaller\smaller\smaller\begin{tabular}{@{}c@{}}%
0\\0\\0
\end{tabular}\endgroup%
}}\!\right]$}%
}%
\ifdim\wd\matricesbox>\halfwidth\myboxwidth=\hsize\else\myboxwidth=\halfwidth\fi
\vbox{%
\ifdim\myboxwidth=\hsize
\setbox\onelinebox=\hbox{%
\vbox{\hbox{%
$\Pi_{4,22}$ spans $L_{22.2}$%
}\hbox{%
$6|62|2\rtimes D_{2}$%
}%
}%
\hfill\copy\matricesbox
}%
\ifdim\wd\onelinebox>\myboxwidth
\hbox to \myboxwidth{%
$\Pi_{4,22}$ spans $L_{22.2}$%
\hfil
$6|62|2\rtimes D_{2}$%
}%
\box\matricesbox
\else
\hbox to \myboxwidth{%
\unhbox\onelinebox
}%
\fi
\else
\hbox to \myboxwidth{%
$\Pi_{4,22}$ spans $L_{22.2}$%
\hfil}%
\hbox to \myboxwidth{%
$6|62|2\rtimes D_{2}$%
\hfil}%
\box\matricesbox
\fi
}%
\hfill\discretionary{}{}{}%
\setbox\matricesbox=\hbox{%
{$\left[\!\llap{\phantom{%
\begingroup \smaller\smaller\smaller\begin{tabular}{@{}c@{}}%
\phantom{0}\\\phantom{0}\\\phantom{0}
\end{tabular}\endgroup%
}}\right.$}%
\begingroup \smaller\smaller\smaller\begin{tabular}{@{}c@{}}%
-1/40\\\phantom{0}\\\phantom{0}
\end{tabular}\endgroup%
\kern3pt%
\begingroup \smaller\smaller\smaller\begin{tabular}{@{}c@{}}%
\phantom{0}\\21/10\\\phantom{0}
\end{tabular}\endgroup%
\kern3pt%
\begingroup \smaller\smaller\smaller\begin{tabular}{@{}c@{}}%
\phantom{0}\\\phantom{0}\\84
\end{tabular}\endgroup%
{$\left.\llap{\phantom{%
\begingroup \smaller\smaller\smaller\begin{tabular}{@{}c@{}}%
\phantom{0}\\\phantom{0}\\\phantom{0}
\end{tabular}\endgroup%
}}\!\right]$}%
{$\left[\!\llap{\phantom{%
\begingroup \smaller\smaller\smaller\begin{tabular}{@{}c@{}}%
0\\0\\0
\end{tabular}\endgroup%
}}\right.$}%
\begingroup \smaller\smaller\smaller\begin{tabular}{@{}c@{}}%
14\\-3\\0
\end{tabular}\endgroup%
\kern3pt%
\begingroup \smaller\smaller\smaller\begin{tabular}{@{}c@{}}%
42\\1\\1
\end{tabular}\endgroup%
\kern3pt%
\begingroup \smaller\smaller\smaller\begin{tabular}{@{}c@{}}%
2\\1\\0
\end{tabular}\endgroup%
{$\left.\llap{\phantom{%
\begingroup \smaller\smaller\smaller\begin{tabular}{@{}c@{}}%
0\\0\\0
\end{tabular}\endgroup%
}}\!\right]$}%
}%
\ifdim\wd\matricesbox>\halfwidth\myboxwidth=\hsize\else\myboxwidth=\halfwidth\fi
\vbox{%
\ifdim\myboxwidth=\hsize
\setbox\onelinebox=\hbox{%
\vbox{\hbox{%
$\Pi_{4,23}$ spans $L_{22.10}$%
}\hbox{%
$6|62|2\rtimes D_{2}$%
}%
}%
\hfill\copy\matricesbox
}%
\ifdim\wd\onelinebox>\myboxwidth
\hbox to \myboxwidth{%
$\Pi_{4,23}$ spans $L_{22.10}$%
\hfil
$6|62|2\rtimes D_{2}$%
}%
\box\matricesbox
\else
\hbox to \myboxwidth{%
\unhbox\onelinebox
}%
\fi
\else
\hbox to \myboxwidth{%
$\Pi_{4,23}$ spans $L_{22.10}$%
\hfil}%
\hbox to \myboxwidth{%
$6|62|2\rtimes D_{2}$%
\hfil}%
\box\matricesbox
\fi
}%
\hfill\discretionary{}{}{}%
\setbox\matricesbox=\hbox{%
{$\left[\!\llap{\phantom{%
\begingroup \smaller\smaller\smaller\begin{tabular}{@{}c@{}}%
\phantom{0}\\\phantom{0}\\\phantom{0}
\end{tabular}\endgroup%
}}\right.$}%
\begingroup \smaller\smaller\smaller\begin{tabular}{@{}c@{}}%
-1/5\\\phantom{0}\\\phantom{0}
\end{tabular}\endgroup%
\kern3pt%
\begingroup \smaller\smaller\smaller\begin{tabular}{@{}c@{}}%
\phantom{0}\\6/5\\\phantom{0}
\end{tabular}\endgroup%
\kern3pt%
\begingroup \smaller\smaller\smaller\begin{tabular}{@{}c@{}}%
\phantom{0}\\\phantom{0}\\6
\end{tabular}\endgroup%
{$\left.\llap{\phantom{%
\begingroup \smaller\smaller\smaller\begin{tabular}{@{}c@{}}%
\phantom{0}\\\phantom{0}\\\phantom{0}
\end{tabular}\endgroup%
}}\!\right]$}%
{$\left[\!\llap{\phantom{%
\begingroup \smaller\smaller\smaller\begin{tabular}{@{}c@{}}%
0\\0\\0
\end{tabular}\endgroup%
}}\right.$}%
\begingroup \smaller\smaller\smaller\begin{tabular}{@{}c@{}}%
1\\-1\\0
\end{tabular}\endgroup%
\kern3pt%
\begingroup \smaller\smaller\smaller\begin{tabular}{@{}c@{}}%
4\\1\\1
\end{tabular}\endgroup%
\kern3pt%
\begingroup \smaller\smaller\smaller\begin{tabular}{@{}c@{}}%
3\\2\\0
\end{tabular}\endgroup%
{$\left.\llap{\phantom{%
\begingroup \smaller\smaller\smaller\begin{tabular}{@{}c@{}}%
0\\0\\0
\end{tabular}\endgroup%
}}\!\right]$}%
}%
\ifdim\wd\matricesbox>\halfwidth\myboxwidth=\hsize\else\myboxwidth=\halfwidth\fi
\vbox{%
\ifdim\myboxwidth=\hsize
\setbox\onelinebox=\hbox{%
\vbox{\hbox{%
$\Pi_{4,24}$ spans $L_{4.18}$%
}\hbox{%
$\infty|\infty2|2\rtimes D_{2}$ (shared)%
}%
}%
\hfill\copy\matricesbox
}%
\ifdim\wd\onelinebox>\myboxwidth
\hbox to \myboxwidth{%
$\Pi_{4,24}$ spans $L_{4.18}$%
\hfil
$\infty|\infty2|2\rtimes D_{2}$ (shared)%
}%
\box\matricesbox
\else
\hbox to \myboxwidth{%
\unhbox\onelinebox
}%
\fi
\else
\hbox to \myboxwidth{%
$\Pi_{4,24}$ spans $L_{4.18}$%
\hfil}%
\hbox to \myboxwidth{%
$\infty|\infty2|2\rtimes D_{2}$ (shared)%
\hfil}%
\box\matricesbox
\fi
}%
\hfill\discretionary{}{}{}%
\setbox\matricesbox=\hbox{%
{$\left[\!\llap{\phantom{%
\begingroup \smaller\smaller\smaller\begin{tabular}{@{}c@{}}%
\phantom{0}\\\phantom{0}\\\phantom{0}
\end{tabular}\endgroup%
}}\right.$}%
\begingroup \smaller\smaller\smaller\begin{tabular}{@{}c@{}}%
-1/7\\\phantom{0}\\\phantom{0}
\end{tabular}\endgroup%
\kern3pt%
\begingroup \smaller\smaller\smaller\begin{tabular}{@{}c@{}}%
\phantom{0}\\2/7\\\phantom{0}
\end{tabular}\endgroup%
\kern3pt%
\begingroup \smaller\smaller\smaller\begin{tabular}{@{}c@{}}%
\phantom{0}\\\phantom{0}\\16
\end{tabular}\endgroup%
{$\left.\llap{\phantom{%
\begingroup \smaller\smaller\smaller\begin{tabular}{@{}c@{}}%
\phantom{0}\\\phantom{0}\\\phantom{0}
\end{tabular}\endgroup%
}}\!\right]$}%
{$\left[\!\llap{\phantom{%
\begingroup \smaller\smaller\smaller\begin{tabular}{@{}c@{}}%
0\\0\\0
\end{tabular}\endgroup%
}}\right.$}%
\begingroup \smaller\smaller\smaller\begin{tabular}{@{}c@{}}%
2\\-3\\0
\end{tabular}\endgroup%
\kern3pt%
\begingroup \smaller\smaller\smaller\begin{tabular}{@{}c@{}}%
8\\2\\1
\end{tabular}\endgroup%
\kern3pt%
\begingroup \smaller\smaller\smaller\begin{tabular}{@{}c@{}}%
1\\2\\0
\end{tabular}\endgroup%
{$\left.\llap{\phantom{%
\begingroup \smaller\smaller\smaller\begin{tabular}{@{}c@{}}%
0\\0\\0
\end{tabular}\endgroup%
}}\!\right]$}%
}%
\ifdim\wd\matricesbox>\halfwidth\myboxwidth=\hsize\else\myboxwidth=\halfwidth\fi
\vbox{%
\ifdim\myboxwidth=\hsize
\setbox\onelinebox=\hbox{%
\vbox{\hbox{%
$\Pi_{4,25}$ spans $L_{141.11}$%
}\hbox{%
$\infty|\infty2|2\rtimes D_{2}$ (shared)%
}%
}%
\hfill\copy\matricesbox
}%
\ifdim\wd\onelinebox>\myboxwidth
\hbox to \myboxwidth{%
$\Pi_{4,25}$ spans $L_{141.11}$%
\hfil
$\infty|\infty2|2\rtimes D_{2}$ (shared)%
}%
\box\matricesbox
\else
\hbox to \myboxwidth{%
\unhbox\onelinebox
}%
\fi
\else
\hbox to \myboxwidth{%
$\Pi_{4,25}$ spans $L_{141.11}$%
\hfil}%
\hbox to \myboxwidth{%
$\infty|\infty2|2\rtimes D_{2}$ (shared)%
\hfil}%
\box\matricesbox
\fi
}%
\hfill\discretionary{}{}{}%
\setbox\matricesbox=\hbox{%
{$\left[\!\llap{\phantom{%
\begingroup \smaller\smaller\smaller\begin{tabular}{@{}c@{}}%
\phantom{0}\\\phantom{0}\\\phantom{0}
\end{tabular}\endgroup%
}}\right.$}%
\begingroup \smaller\smaller\smaller\begin{tabular}{@{}c@{}}%
-1/40\\\phantom{0}\\\phantom{0}
\end{tabular}\endgroup%
\kern3pt%
\begingroup \smaller\smaller\smaller\begin{tabular}{@{}c@{}}%
\phantom{0}\\3/5\\\phantom{0}
\end{tabular}\endgroup%
\kern3pt%
\begingroup \smaller\smaller\smaller\begin{tabular}{@{}c@{}}%
\phantom{0}\\\phantom{0}\\3/2
\end{tabular}\endgroup%
{$\left.\llap{\phantom{%
\begingroup \smaller\smaller\smaller\begin{tabular}{@{}c@{}}%
\phantom{0}\\\phantom{0}\\\phantom{0}
\end{tabular}\endgroup%
}}\!\right]$}%
{$\left[\!\llap{\phantom{%
\begingroup \smaller\smaller\smaller\begin{tabular}{@{}c@{}}%
0\\0\\0
\end{tabular}\endgroup%
}}\right.$}%
\begingroup \smaller\smaller\smaller\begin{tabular}{@{}c@{}}%
2\\-1\\1
\end{tabular}\endgroup%
\kern3pt%
\begingroup \smaller\smaller\smaller\begin{tabular}{@{}c@{}}%
12\\4\\2
\end{tabular}\endgroup%
{$\left.\llap{\phantom{%
\begingroup \smaller\smaller\smaller\begin{tabular}{@{}c@{}}%
0\\0\\0
\end{tabular}\endgroup%
}}\!\right]$}%
}%
\ifdim\wd\matricesbox>\halfwidth\myboxwidth=\hsize\else\myboxwidth=\halfwidth\fi
\vbox{%
\ifdim\myboxwidth=\hsize
\setbox\onelinebox=\hbox{%
\vbox{\hbox{%
$\Pi_{4,26}$ spans $L_{3.4}$%
}\hbox{%
$\slashthree2\slashtwo2\rtimes D_{2}$ (shared)%
}%
}%
\hfill\copy\matricesbox
}%
\ifdim\wd\onelinebox>\myboxwidth
\hbox to \myboxwidth{%
$\Pi_{4,26}$ spans $L_{3.4}$%
\hfil
$\slashthree2\slashtwo2\rtimes D_{2}$ (shared)%
}%
\box\matricesbox
\else
\hbox to \myboxwidth{%
\unhbox\onelinebox
}%
\fi
\else
\hbox to \myboxwidth{%
$\Pi_{4,26}$ spans $L_{3.4}$%
\hfil}%
\hbox to \myboxwidth{%
$\slashthree2\slashtwo2\rtimes D_{2}$ (shared)%
\hfil}%
\box\matricesbox
\fi
}%
\hfill\discretionary{}{}{}%
\setbox\matricesbox=\hbox{%
{$\left[\!\llap{\phantom{%
\begingroup \smaller\smaller\smaller\begin{tabular}{@{}c@{}}%
\phantom{0}\\\phantom{0}\\\phantom{0}
\end{tabular}\endgroup%
}}\right.$}%
\begingroup \smaller\smaller\smaller\begin{tabular}{@{}c@{}}%
-1/28\\\phantom{0}\\\phantom{0}
\end{tabular}\endgroup%
\kern3pt%
\begingroup \smaller\smaller\smaller\begin{tabular}{@{}c@{}}%
\phantom{0}\\9/14\\\phantom{0}
\end{tabular}\endgroup%
\kern3pt%
\begingroup \smaller\smaller\smaller\begin{tabular}{@{}c@{}}%
\phantom{0}\\\phantom{0}\\3/2
\end{tabular}\endgroup%
{$\left.\llap{\phantom{%
\begingroup \smaller\smaller\smaller\begin{tabular}{@{}c@{}}%
\phantom{0}\\\phantom{0}\\\phantom{0}
\end{tabular}\endgroup%
}}\!\right]$}%
{$\left[\!\llap{\phantom{%
\begingroup \smaller\smaller\smaller\begin{tabular}{@{}c@{}}%
0\\0\\0
\end{tabular}\endgroup%
}}\right.$}%
\begingroup \smaller\smaller\smaller\begin{tabular}{@{}c@{}}%
2\\1\\1
\end{tabular}\endgroup%
\kern3pt%
\begingroup \smaller\smaller\smaller\begin{tabular}{@{}c@{}}%
18\\-5\\3
\end{tabular}\endgroup%
{$\left.\llap{\phantom{%
\begingroup \smaller\smaller\smaller\begin{tabular}{@{}c@{}}%
0\\0\\0
\end{tabular}\endgroup%
}}\!\right]$}%
}%
\ifdim\wd\matricesbox>\halfwidth\myboxwidth=\hsize\else\myboxwidth=\halfwidth\fi
\vbox{%
\ifdim\myboxwidth=\hsize
\setbox\onelinebox=\hbox{%
\vbox{\hbox{%
$\Pi_{4,27}=\hbox{GN}_{8}$ spans $L_{155.1}$%
}\hbox{%
$\slashthree2\slashthree2\rtimes D_{2}$ (shared)%
}%
}%
\hfill\copy\matricesbox
}%
\ifdim\wd\onelinebox>\myboxwidth
\hbox to \myboxwidth{%
$\Pi_{4,27}=\hbox{GN}_{8}$ spans $L_{155.1}$%
\hfil
$\slashthree2\slashthree2\rtimes D_{2}$ (shared)%
}%
\box\matricesbox
\else
\hbox to \myboxwidth{%
\unhbox\onelinebox
}%
\fi
\else
\hbox to \myboxwidth{%
$\Pi_{4,27}=\hbox{GN}_{8}$ spans $L_{155.1}$%
\hfil}%
\hbox to \myboxwidth{%
$\slashthree2\slashthree2\rtimes D_{2}$ (shared)%
\hfil}%
\box\matricesbox
\fi
}%
\hfill\discretionary{}{}{}%
\setbox\matricesbox=\hbox{%
{$\left[\!\llap{\phantom{%
\begingroup \smaller\smaller\smaller\begin{tabular}{@{}c@{}}%
\phantom{0}\\\phantom{0}\\\phantom{0}
\end{tabular}\endgroup%
}}\right.$}%
\begingroup \smaller\smaller\smaller\begin{tabular}{@{}c@{}}%
-1/24\\\phantom{0}\\\phantom{0}
\end{tabular}\endgroup%
\kern3pt%
\begingroup \smaller\smaller\smaller\begin{tabular}{@{}c@{}}%
\phantom{0}\\2/3\\\phantom{0}
\end{tabular}\endgroup%
\kern3pt%
\begingroup \smaller\smaller\smaller\begin{tabular}{@{}c@{}}%
\phantom{0}\\\phantom{0}\\3/2
\end{tabular}\endgroup%
{$\left.\llap{\phantom{%
\begingroup \smaller\smaller\smaller\begin{tabular}{@{}c@{}}%
\phantom{0}\\\phantom{0}\\\phantom{0}
\end{tabular}\endgroup%
}}\!\right]$}%
{$\left[\!\llap{\phantom{%
\begingroup \smaller\smaller\smaller\begin{tabular}{@{}c@{}}%
0\\0\\0
\end{tabular}\endgroup%
}}\right.$}%
\begingroup \smaller\smaller\smaller\begin{tabular}{@{}c@{}}%
2\\-1\\1
\end{tabular}\endgroup%
\kern3pt%
\begingroup \smaller\smaller\smaller\begin{tabular}{@{}c@{}}%
24\\6\\4
\end{tabular}\endgroup%
{$\left.\llap{\phantom{%
\begingroup \smaller\smaller\smaller\begin{tabular}{@{}c@{}}%
0\\0\\0
\end{tabular}\endgroup%
}}\!\right]$}%
}%
\ifdim\wd\matricesbox>\halfwidth\myboxwidth=\hsize\else\myboxwidth=\halfwidth\fi
\vbox{%
\ifdim\myboxwidth=\hsize
\setbox\onelinebox=\hbox{%
\vbox{\hbox{%
$\Pi_{4,28}$ spans $L_{7.13}$%
}\hbox{%
$\slashthree2\slashinfty2\rtimes D_{2}$ (shared)%
}%
}%
\hfill\copy\matricesbox
}%
\ifdim\wd\onelinebox>\myboxwidth
\hbox to \myboxwidth{%
$\Pi_{4,28}$ spans $L_{7.13}$%
\hfil
$\slashthree2\slashinfty2\rtimes D_{2}$ (shared)%
}%
\box\matricesbox
\else
\hbox to \myboxwidth{%
\unhbox\onelinebox
}%
\fi
\else
\hbox to \myboxwidth{%
$\Pi_{4,28}$ spans $L_{7.13}$%
\hfil}%
\hbox to \myboxwidth{%
$\slashthree2\slashinfty2\rtimes D_{2}$ (shared)%
\hfil}%
\box\matricesbox
\fi
}%
\hfill\discretionary{}{}{}%
\setbox\matricesbox=\hbox{%
{$\left[\!\llap{\phantom{%
\begingroup \smaller\smaller\smaller\begin{tabular}{@{}c@{}}%
\phantom{0}\\\phantom{0}\\\phantom{0}
\end{tabular}\endgroup%
}}\right.$}%
\begingroup \smaller\smaller\smaller\begin{tabular}{@{}c@{}}%
-1/19\\\phantom{0}\\\phantom{0}
\end{tabular}\endgroup%
\kern3pt%
\begingroup \smaller\smaller\smaller\begin{tabular}{@{}c@{}}%
\phantom{0}\\3/38\\\phantom{0}
\end{tabular}\endgroup%
\kern3pt%
\begingroup \smaller\smaller\smaller\begin{tabular}{@{}c@{}}%
\phantom{0}\\\phantom{0}\\3/2
\end{tabular}\endgroup%
{$\left.\llap{\phantom{%
\begingroup \smaller\smaller\smaller\begin{tabular}{@{}c@{}}%
\phantom{0}\\\phantom{0}\\\phantom{0}
\end{tabular}\endgroup%
}}\!\right]$}%
{$\left[\!\llap{\phantom{%
\begingroup \smaller\smaller\smaller\begin{tabular}{@{}c@{}}%
0\\0\\0
\end{tabular}\endgroup%
}}\right.$}%
\begingroup \smaller\smaller\smaller\begin{tabular}{@{}c@{}}%
2\\3\\-1
\end{tabular}\endgroup%
\kern3pt%
\begingroup \smaller\smaller\smaller\begin{tabular}{@{}c@{}}%
3\\-5\\-1
\end{tabular}\endgroup%
{$\left.\llap{\phantom{%
\begingroup \smaller\smaller\smaller\begin{tabular}{@{}c@{}}%
0\\0\\0
\end{tabular}\endgroup%
}}\!\right]$}%
}%
\ifdim\wd\matricesbox>\halfwidth\myboxwidth=\hsize\else\myboxwidth=\halfwidth\fi
\vbox{%
\ifdim\myboxwidth=\hsize
\setbox\onelinebox=\hbox{%
\vbox{\hbox{%
$\Pi_{4,29}$ spans $L_{3.3}$%
}\hbox{%
$\slashthree2\slashtwo2\rtimes D_{2}$ (shared)%
}%
}%
\hfill\copy\matricesbox
}%
\ifdim\wd\onelinebox>\myboxwidth
\hbox to \myboxwidth{%
$\Pi_{4,29}$ spans $L_{3.3}$%
\hfil
$\slashthree2\slashtwo2\rtimes D_{2}$ (shared)%
}%
\box\matricesbox
\else
\hbox to \myboxwidth{%
\unhbox\onelinebox
}%
\fi
\else
\hbox to \myboxwidth{%
$\Pi_{4,29}$ spans $L_{3.3}$%
\hfil}%
\hbox to \myboxwidth{%
$\slashthree2\slashtwo2\rtimes D_{2}$ (shared)%
\hfil}%
\box\matricesbox
\fi
}%
\hfill\discretionary{}{}{}%
\setbox\matricesbox=\hbox{%
{$\left[\!\llap{\phantom{%
\begingroup \smaller\smaller\smaller\begin{tabular}{@{}c@{}}%
\phantom{0}\\\phantom{0}\\\phantom{0}
\end{tabular}\endgroup%
}}\right.$}%
\begingroup \smaller\smaller\smaller\begin{tabular}{@{}c@{}}%
-1/10\\\phantom{0}\\\phantom{0}
\end{tabular}\endgroup%
\kern3pt%
\begingroup \smaller\smaller\smaller\begin{tabular}{@{}c@{}}%
\phantom{0}\\9/10\\\phantom{0}
\end{tabular}\endgroup%
\kern3pt%
\begingroup \smaller\smaller\smaller\begin{tabular}{@{}c@{}}%
\phantom{0}\\\phantom{0}\\3/2
\end{tabular}\endgroup%
{$\left.\llap{\phantom{%
\begingroup \smaller\smaller\smaller\begin{tabular}{@{}c@{}}%
\phantom{0}\\\phantom{0}\\\phantom{0}
\end{tabular}\endgroup%
}}\!\right]$}%
{$\left[\!\llap{\phantom{%
\begingroup \smaller\smaller\smaller\begin{tabular}{@{}c@{}}%
0\\0\\0
\end{tabular}\endgroup%
}}\right.$}%
\begingroup \smaller\smaller\smaller\begin{tabular}{@{}c@{}}%
2\\1\\1
\end{tabular}\endgroup%
\kern3pt%
\begingroup \smaller\smaller\smaller\begin{tabular}{@{}c@{}}%
6\\-2\\2
\end{tabular}\endgroup%
{$\left.\llap{\phantom{%
\begingroup \smaller\smaller\smaller\begin{tabular}{@{}c@{}}%
0\\0\\0
\end{tabular}\endgroup%
}}\!\right]$}%
}%
\ifdim\wd\matricesbox>\halfwidth\myboxwidth=\hsize\else\myboxwidth=\halfwidth\fi
\vbox{%
\ifdim\myboxwidth=\hsize
\setbox\onelinebox=\hbox{%
\vbox{\hbox{%
$\Pi_{4,30}$ spans $L_{7.8}$%
}\hbox{%
$\slashthree2\slashinfty2\rtimes D_{2}$ (shared)%
}%
}%
\hfill\copy\matricesbox
}%
\ifdim\wd\onelinebox>\myboxwidth
\hbox to \myboxwidth{%
$\Pi_{4,30}$ spans $L_{7.8}$%
\hfil
$\slashthree2\slashinfty2\rtimes D_{2}$ (shared)%
}%
\box\matricesbox
\else
\hbox to \myboxwidth{%
\unhbox\onelinebox
}%
\fi
\else
\hbox to \myboxwidth{%
$\Pi_{4,30}$ spans $L_{7.8}$%
\hfil}%
\hbox to \myboxwidth{%
$\slashthree2\slashinfty2\rtimes D_{2}$ (shared)%
\hfil}%
\box\matricesbox
\fi
}%
\hfill\discretionary{}{}{}%
\setbox\matricesbox=\hbox{%
{$\left[\!\llap{\phantom{%
\begingroup \smaller\smaller\smaller\begin{tabular}{@{}c@{}}%
\phantom{0}\\\phantom{0}\\\phantom{0}
\end{tabular}\endgroup%
}}\right.$}%
\begingroup \smaller\smaller\smaller\begin{tabular}{@{}c@{}}%
-1/40\\\phantom{0}\\\phantom{0}
\end{tabular}\endgroup%
\kern3pt%
\begingroup \smaller\smaller\smaller\begin{tabular}{@{}c@{}}%
\phantom{0}\\12/5\\\phantom{0}
\end{tabular}\endgroup%
\kern3pt%
\begingroup \smaller\smaller\smaller\begin{tabular}{@{}c@{}}%
\phantom{0}\\\phantom{0}\\1/2
\end{tabular}\endgroup%
{$\left.\llap{\phantom{%
\begingroup \smaller\smaller\smaller\begin{tabular}{@{}c@{}}%
\phantom{0}\\\phantom{0}\\\phantom{0}
\end{tabular}\endgroup%
}}\!\right]$}%
{$\left[\!\llap{\phantom{%
\begingroup \smaller\smaller\smaller\begin{tabular}{@{}c@{}}%
0\\0\\0
\end{tabular}\endgroup%
}}\right.$}%
\begingroup \smaller\smaller\smaller\begin{tabular}{@{}c@{}}%
6\\-1\\3
\end{tabular}\endgroup%
\kern3pt%
\begingroup \smaller\smaller\smaller\begin{tabular}{@{}c@{}}%
4\\1\\2
\end{tabular}\endgroup%
{$\left.\llap{\phantom{%
\begingroup \smaller\smaller\smaller\begin{tabular}{@{}c@{}}%
0\\0\\0
\end{tabular}\endgroup%
}}\!\right]$}%
}%
\ifdim\wd\matricesbox>\halfwidth\myboxwidth=\hsize\else\myboxwidth=\halfwidth\fi
\vbox{%
\ifdim\myboxwidth=\hsize
\setbox\onelinebox=\hbox{%
\vbox{\hbox{%
$\Pi_{4,31}$ spans $L_{3.1}$%
}\hbox{%
$\slashthree2\slashtwo2\rtimes D_{2}$ (shared)%
}%
}%
\hfill\copy\matricesbox
}%
\ifdim\wd\onelinebox>\myboxwidth
\hbox to \myboxwidth{%
$\Pi_{4,31}$ spans $L_{3.1}$%
\hfil
$\slashthree2\slashtwo2\rtimes D_{2}$ (shared)%
}%
\box\matricesbox
\else
\hbox to \myboxwidth{%
\unhbox\onelinebox
}%
\fi
\else
\hbox to \myboxwidth{%
$\Pi_{4,31}$ spans $L_{3.1}$%
\hfil}%
\hbox to \myboxwidth{%
$\slashthree2\slashtwo2\rtimes D_{2}$ (shared)%
\hfil}%
\box\matricesbox
\fi
}%
\hfill\discretionary{}{}{}%
\setbox\matricesbox=\hbox{%
{$\left[\!\llap{\phantom{%
\begingroup \smaller\smaller\smaller\begin{tabular}{@{}c@{}}%
\phantom{0}\\\phantom{0}\\\phantom{0}
\end{tabular}\endgroup%
}}\right.$}%
\begingroup \smaller\smaller\smaller\begin{tabular}{@{}c@{}}%
-3/56\\\phantom{0}\\\phantom{0}
\end{tabular}\endgroup%
\kern3pt%
\begingroup \smaller\smaller\smaller\begin{tabular}{@{}c@{}}%
\phantom{0}\\24/7\\\phantom{0}
\end{tabular}\endgroup%
\kern3pt%
\begingroup \smaller\smaller\smaller\begin{tabular}{@{}c@{}}%
\phantom{0}\\\phantom{0}\\1/2
\end{tabular}\endgroup%
{$\left.\llap{\phantom{%
\begingroup \smaller\smaller\smaller\begin{tabular}{@{}c@{}}%
\phantom{0}\\\phantom{0}\\\phantom{0}
\end{tabular}\endgroup%
}}\!\right]$}%
{$\left[\!\llap{\phantom{%
\begingroup \smaller\smaller\smaller\begin{tabular}{@{}c@{}}%
0\\0\\0
\end{tabular}\endgroup%
}}\right.$}%
\begingroup \smaller\smaller\smaller\begin{tabular}{@{}c@{}}%
6\\-1\\-3
\end{tabular}\endgroup%
\kern3pt%
\begingroup \smaller\smaller\smaller\begin{tabular}{@{}c@{}}%
8\\1\\-4
\end{tabular}\endgroup%
{$\left.\llap{\phantom{%
\begingroup \smaller\smaller\smaller\begin{tabular}{@{}c@{}}%
0\\0\\0
\end{tabular}\endgroup%
}}\!\right]$}%
}%
\ifdim\wd\matricesbox>\halfwidth\myboxwidth=\hsize\else\myboxwidth=\halfwidth\fi
\vbox{%
\ifdim\myboxwidth=\hsize
\setbox\onelinebox=\hbox{%
\vbox{\hbox{%
$\Pi_{4,32}$ spans $L_{7.9}$%
}\hbox{%
$\slashthree2\slashinfty2\rtimes D_{2}$ (shared)%
}%
}%
\hfill\copy\matricesbox
}%
\ifdim\wd\onelinebox>\myboxwidth
\hbox to \myboxwidth{%
$\Pi_{4,32}$ spans $L_{7.9}$%
\hfil
$\slashthree2\slashinfty2\rtimes D_{2}$ (shared)%
}%
\box\matricesbox
\else
\hbox to \myboxwidth{%
\unhbox\onelinebox
}%
\fi
\else
\hbox to \myboxwidth{%
$\Pi_{4,32}$ spans $L_{7.9}$%
\hfil}%
\hbox to \myboxwidth{%
$\slashthree2\slashinfty2\rtimes D_{2}$ (shared)%
\hfil}%
\box\matricesbox
\fi
}%
\hfill\discretionary{}{}{}%
\setbox\matricesbox=\hbox{%
{$\left[\!\llap{\phantom{%
\begingroup \smaller\smaller\smaller\begin{tabular}{@{}c@{}}%
\phantom{0}\\\phantom{0}\\\phantom{0}
\end{tabular}\endgroup%
}}\right.$}%
\begingroup \smaller\smaller\smaller\begin{tabular}{@{}c@{}}%
-1/9\\\phantom{0}\\\phantom{0}
\end{tabular}\endgroup%
\kern3pt%
\begingroup \smaller\smaller\smaller\begin{tabular}{@{}c@{}}%
\phantom{0}\\1/9\\\phantom{0}
\end{tabular}\endgroup%
\kern3pt%
\begingroup \smaller\smaller\smaller\begin{tabular}{@{}c@{}}%
\phantom{0}\\\phantom{0}\\3
\end{tabular}\endgroup%
{$\left.\llap{\phantom{%
\begingroup \smaller\smaller\smaller\begin{tabular}{@{}c@{}}%
\phantom{0}\\\phantom{0}\\\phantom{0}
\end{tabular}\endgroup%
}}\!\right]$}%
{$\left[\!\llap{\phantom{%
\begingroup \smaller\smaller\smaller\begin{tabular}{@{}c@{}}%
0\\0\\0
\end{tabular}\endgroup%
}}\right.$}%
\begingroup \smaller\smaller\smaller\begin{tabular}{@{}c@{}}%
4\\-5\\-1
\end{tabular}\endgroup%
\kern3pt%
\begingroup \smaller\smaller\smaller\begin{tabular}{@{}c@{}}%
3\\3\\-1
\end{tabular}\endgroup%
{$\left.\llap{\phantom{%
\begingroup \smaller\smaller\smaller\begin{tabular}{@{}c@{}}%
0\\0\\0
\end{tabular}\endgroup%
}}\!\right]$}%
}%
\ifdim\wd\matricesbox>\halfwidth\myboxwidth=\hsize\else\myboxwidth=\halfwidth\fi
\vbox{%
\ifdim\myboxwidth=\hsize
\setbox\onelinebox=\hbox{%
\vbox{\hbox{%
$\Pi_{4,33}$ spans $L_{7.11}$%
}\hbox{%
$\slashthree2\slashinfty2\rtimes D_{2}$ (shared)%
}%
}%
\hfill\copy\matricesbox
}%
\ifdim\wd\onelinebox>\myboxwidth
\hbox to \myboxwidth{%
$\Pi_{4,33}$ spans $L_{7.11}$%
\hfil
$\slashthree2\slashinfty2\rtimes D_{2}$ (shared)%
}%
\box\matricesbox
\else
\hbox to \myboxwidth{%
\unhbox\onelinebox
}%
\fi
\else
\hbox to \myboxwidth{%
$\Pi_{4,33}$ spans $L_{7.11}$%
\hfil}%
\hbox to \myboxwidth{%
$\slashthree2\slashinfty2\rtimes D_{2}$ (shared)%
\hfil}%
\box\matricesbox
\fi
}%
\hfill\discretionary{}{}{}%
\setbox\matricesbox=\hbox{%
{$\left[\!\llap{\phantom{%
\begingroup \smaller\smaller\smaller\begin{tabular}{@{}c@{}}%
\phantom{0}\\\phantom{0}\\\phantom{0}
\end{tabular}\endgroup%
}}\right.$}%
\begingroup \smaller\smaller\smaller\begin{tabular}{@{}c@{}}%
-1/25\\\phantom{0}\\\phantom{0}
\end{tabular}\endgroup%
\kern3pt%
\begingroup \smaller\smaller\smaller\begin{tabular}{@{}c@{}}%
\phantom{0}\\3/50\\\phantom{0}
\end{tabular}\endgroup%
\kern3pt%
\begingroup \smaller\smaller\smaller\begin{tabular}{@{}c@{}}%
\phantom{0}\\\phantom{0}\\1/2
\end{tabular}\endgroup%
{$\left.\llap{\phantom{%
\begingroup \smaller\smaller\smaller\begin{tabular}{@{}c@{}}%
\phantom{0}\\\phantom{0}\\\phantom{0}
\end{tabular}\endgroup%
}}\!\right]$}%
{$\left[\!\llap{\phantom{%
\begingroup \smaller\smaller\smaller\begin{tabular}{@{}c@{}}%
0\\0\\0
\end{tabular}\endgroup%
}}\right.$}%
\begingroup \smaller\smaller\smaller\begin{tabular}{@{}c@{}}%
6\\7\\3
\end{tabular}\endgroup%
\kern3pt%
\begingroup \smaller\smaller\smaller\begin{tabular}{@{}c@{}}%
1\\-3\\1
\end{tabular}\endgroup%
{$\left.\llap{\phantom{%
\begingroup \smaller\smaller\smaller\begin{tabular}{@{}c@{}}%
0\\0\\0
\end{tabular}\endgroup%
}}\!\right]$}%
}%
\ifdim\wd\matricesbox>\halfwidth\myboxwidth=\hsize\else\myboxwidth=\halfwidth\fi
\vbox{%
\ifdim\myboxwidth=\hsize
\setbox\onelinebox=\hbox{%
\vbox{\hbox{%
$\Pi_{4,34}$ spans $L_{3.2}$%
}\hbox{%
$\slashthree2\slashtwo2\rtimes D_{2}$ (shared)%
}%
}%
\hfill\copy\matricesbox
}%
\ifdim\wd\onelinebox>\myboxwidth
\hbox to \myboxwidth{%
$\Pi_{4,34}$ spans $L_{3.2}$%
\hfil
$\slashthree2\slashtwo2\rtimes D_{2}$ (shared)%
}%
\box\matricesbox
\else
\hbox to \myboxwidth{%
\unhbox\onelinebox
}%
\fi
\else
\hbox to \myboxwidth{%
$\Pi_{4,34}$ spans $L_{3.2}$%
\hfil}%
\hbox to \myboxwidth{%
$\slashthree2\slashtwo2\rtimes D_{2}$ (shared)%
\hfil}%
\box\matricesbox
\fi
}%
\hfill\discretionary{}{}{}%
\setbox\matricesbox=\hbox{%
{$\left[\!\llap{\phantom{%
\begingroup \smaller\smaller\smaller\begin{tabular}{@{}c@{}}%
\phantom{0}\\\phantom{0}\\\phantom{0}
\end{tabular}\endgroup%
}}\right.$}%
\begingroup \smaller\smaller\smaller\begin{tabular}{@{}c@{}}%
-3/26\\\phantom{0}\\\phantom{0}
\end{tabular}\endgroup%
\kern3pt%
\begingroup \smaller\smaller\smaller\begin{tabular}{@{}c@{}}%
\phantom{0}\\3/26\\\phantom{0}
\end{tabular}\endgroup%
\kern3pt%
\begingroup \smaller\smaller\smaller\begin{tabular}{@{}c@{}}%
\phantom{0}\\\phantom{0}\\1/2
\end{tabular}\endgroup%
{$\left.\llap{\phantom{%
\begingroup \smaller\smaller\smaller\begin{tabular}{@{}c@{}}%
\phantom{0}\\\phantom{0}\\\phantom{0}
\end{tabular}\endgroup%
}}\!\right]$}%
{$\left[\!\llap{\phantom{%
\begingroup \smaller\smaller\smaller\begin{tabular}{@{}c@{}}%
0\\0\\0
\end{tabular}\endgroup%
}}\right.$}%
\begingroup \smaller\smaller\smaller\begin{tabular}{@{}c@{}}%
6\\-7\\3
\end{tabular}\endgroup%
\kern3pt%
\begingroup \smaller\smaller\smaller\begin{tabular}{@{}c@{}}%
2\\2\\2
\end{tabular}\endgroup%
{$\left.\llap{\phantom{%
\begingroup \smaller\smaller\smaller\begin{tabular}{@{}c@{}}%
0\\0\\0
\end{tabular}\endgroup%
}}\!\right]$}%
}%
\ifdim\wd\matricesbox>\halfwidth\myboxwidth=\hsize\else\myboxwidth=\halfwidth\fi
\vbox{%
\ifdim\myboxwidth=\hsize
\setbox\onelinebox=\hbox{%
\vbox{\hbox{%
$\Pi_{4,35}$ spans $L_{7.6}$%
}\hbox{%
$\slashthree2\slashinfty2\rtimes D_{2}$ (shared)%
}%
}%
\hfill\copy\matricesbox
}%
\ifdim\wd\onelinebox>\myboxwidth
\hbox to \myboxwidth{%
$\Pi_{4,35}$ spans $L_{7.6}$%
\hfil
$\slashthree2\slashinfty2\rtimes D_{2}$ (shared)%
}%
\box\matricesbox
\else
\hbox to \myboxwidth{%
\unhbox\onelinebox
}%
\fi
\else
\hbox to \myboxwidth{%
$\Pi_{4,35}$ spans $L_{7.6}$%
\hfil}%
\hbox to \myboxwidth{%
$\slashthree2\slashinfty2\rtimes D_{2}$ (shared)%
\hfil}%
\box\matricesbox
\fi
}%
\hfill\discretionary{}{}{}%
\setbox\matricesbox=\hbox{%
{$\left[\!\llap{\phantom{%
\begingroup \smaller\smaller\smaller\begin{tabular}{@{}c@{}}%
\phantom{0}\\\phantom{0}\\\phantom{0}
\end{tabular}\endgroup%
}}\right.$}%
\begingroup \smaller\smaller\smaller\begin{tabular}{@{}c@{}}%
-3/25\\\phantom{0}\\\phantom{0}
\end{tabular}\endgroup%
\kern3pt%
\begingroup \smaller\smaller\smaller\begin{tabular}{@{}c@{}}%
\phantom{0}\\3/25\\\phantom{0}
\end{tabular}\endgroup%
\kern3pt%
\begingroup \smaller\smaller\smaller\begin{tabular}{@{}c@{}}%
\phantom{0}\\\phantom{0}\\1
\end{tabular}\endgroup%
{$\left.\llap{\phantom{%
\begingroup \smaller\smaller\smaller\begin{tabular}{@{}c@{}}%
\phantom{0}\\\phantom{0}\\\phantom{0}
\end{tabular}\endgroup%
}}\!\right]$}%
{$\left[\!\llap{\phantom{%
\begingroup \smaller\smaller\smaller\begin{tabular}{@{}c@{}}%
0\\0\\0
\end{tabular}\endgroup%
}}\right.$}%
\begingroup \smaller\smaller\smaller\begin{tabular}{@{}c@{}}%
12\\13\\-3
\end{tabular}\endgroup%
\kern3pt%
\begingroup \smaller\smaller\smaller\begin{tabular}{@{}c@{}}%
1\\-1\\-1
\end{tabular}\endgroup%
{$\left.\llap{\phantom{%
\begingroup \smaller\smaller\smaller\begin{tabular}{@{}c@{}}%
0\\0\\0
\end{tabular}\endgroup%
}}\!\right]$}%
}%
\ifdim\wd\matricesbox>\halfwidth\myboxwidth=\hsize\else\myboxwidth=\halfwidth\fi
\vbox{%
\ifdim\myboxwidth=\hsize
\setbox\onelinebox=\hbox{%
\vbox{\hbox{%
$\Pi_{4,36}=A_{3,II,\onebar}$ spans $L_{7.7}$%
}\hbox{%
$\slashthree2\slashinfty2\rtimes D_{2}$ (shared)%
}%
}%
\hfill\copy\matricesbox
}%
\ifdim\wd\onelinebox>\myboxwidth
\hbox to \myboxwidth{%
$\Pi_{4,36}=A_{3,II,\onebar}$ spans $L_{7.7}$%
\hfil
$\slashthree2\slashinfty2\rtimes D_{2}$ (shared)%
}%
\box\matricesbox
\else
\hbox to \myboxwidth{%
\unhbox\onelinebox
}%
\fi
\else
\hbox to \myboxwidth{%
$\Pi_{4,36}=A_{3,II,\onebar}$ spans $L_{7.7}$%
\hfil}%
\hbox to \myboxwidth{%
$\slashthree2\slashinfty2\rtimes D_{2}$ (shared)%
\hfil}%
\box\matricesbox
\fi
}%
\hfill\discretionary{}{}{}%
\setbox\matricesbox=\hbox{%
{$\left[\!\llap{\phantom{%
\begingroup \smaller\smaller\smaller\begin{tabular}{@{}c@{}}%
\phantom{0}\\\phantom{0}\\\phantom{0}
\end{tabular}\endgroup%
}}\right.$}%
\begingroup \smaller\smaller\smaller\begin{tabular}{@{}c@{}}%
-4/17\\\phantom{0}\\\phantom{0}
\end{tabular}\endgroup%
\kern3pt%
\begingroup \smaller\smaller\smaller\begin{tabular}{@{}c@{}}%
\phantom{0}\\1/34\\\phantom{0}
\end{tabular}\endgroup%
\kern3pt%
\begingroup \smaller\smaller\smaller\begin{tabular}{@{}c@{}}%
\phantom{0}\\\phantom{0}\\1/2
\end{tabular}\endgroup%
{$\left.\llap{\phantom{%
\begingroup \smaller\smaller\smaller\begin{tabular}{@{}c@{}}%
\phantom{0}\\\phantom{0}\\\phantom{0}
\end{tabular}\endgroup%
}}\!\right]$}%
{$\left[\!\llap{\phantom{%
\begingroup \smaller\smaller\smaller\begin{tabular}{@{}c@{}}%
0\\0\\0
\end{tabular}\endgroup%
}}\right.$}%
\begingroup \smaller\smaller\smaller\begin{tabular}{@{}c@{}}%
4\\14\\-2
\end{tabular}\endgroup%
\kern3pt%
\begingroup \smaller\smaller\smaller\begin{tabular}{@{}c@{}}%
1\\-5\\-1
\end{tabular}\endgroup%
{$\left.\llap{\phantom{%
\begingroup \smaller\smaller\smaller\begin{tabular}{@{}c@{}}%
0\\0\\0
\end{tabular}\endgroup%
}}\!\right]$}%
}%
\ifdim\wd\matricesbox>\halfwidth\myboxwidth=\hsize\else\myboxwidth=\halfwidth\fi
\vbox{%
\ifdim\myboxwidth=\hsize
\setbox\onelinebox=\hbox{%
\vbox{\hbox{%
$\Pi_{4,37}=\hbox{GN}_{20}$ spans $L_{145.1}$%
}\hbox{%
$\infty\slashtwo\infty\slashtwo\rtimes D_{2}$ (shared)%
}%
}%
\hfill\copy\matricesbox
}%
\ifdim\wd\onelinebox>\myboxwidth
\hbox to \myboxwidth{%
$\Pi_{4,37}=\hbox{GN}_{20}$ spans $L_{145.1}$%
\hfil
$\infty\slashtwo\infty\slashtwo\rtimes D_{2}$ (shared)%
}%
\box\matricesbox
\else
\hbox to \myboxwidth{%
\unhbox\onelinebox
}%
\fi
\else
\hbox to \myboxwidth{%
$\Pi_{4,37}=\hbox{GN}_{20}$ spans $L_{145.1}$%
\hfil}%
\hbox to \myboxwidth{%
$\infty\slashtwo\infty\slashtwo\rtimes D_{2}$ (shared)%
\hfil}%
\box\matricesbox
\fi
}%
\hfill\discretionary{}{}{}%
\setbox\matricesbox=\hbox{%
{$\left[\!\llap{\phantom{%
\begingroup \smaller\smaller\smaller\begin{tabular}{@{}c@{}}%
\phantom{0}\\\phantom{0}\\\phantom{0}
\end{tabular}\endgroup%
}}\right.$}%
\begingroup \smaller\smaller\smaller\begin{tabular}{@{}c@{}}%
-1/9\\\phantom{0}\\\phantom{0}
\end{tabular}\endgroup%
\kern3pt%
\begingroup \smaller\smaller\smaller\begin{tabular}{@{}c@{}}%
\phantom{0}\\1/9\\\phantom{0}
\end{tabular}\endgroup%
\kern3pt%
\begingroup \smaller\smaller\smaller\begin{tabular}{@{}c@{}}%
\phantom{0}\\\phantom{0}\\1
\end{tabular}\endgroup%
{$\left.\llap{\phantom{%
\begingroup \smaller\smaller\smaller\begin{tabular}{@{}c@{}}%
\phantom{0}\\\phantom{0}\\\phantom{0}
\end{tabular}\endgroup%
}}\!\right]$}%
{$\left[\!\llap{\phantom{%
\begingroup \smaller\smaller\smaller\begin{tabular}{@{}c@{}}%
0\\0\\0
\end{tabular}\endgroup%
}}\right.$}%
\begingroup \smaller\smaller\smaller\begin{tabular}{@{}c@{}}%
1\\-1\\1
\end{tabular}\endgroup%
\kern3pt%
\begingroup \smaller\smaller\smaller\begin{tabular}{@{}c@{}}%
8\\10\\2
\end{tabular}\endgroup%
{$\left.\llap{\phantom{%
\begingroup \smaller\smaller\smaller\begin{tabular}{@{}c@{}}%
0\\0\\0
\end{tabular}\endgroup%
}}\!\right]$}%
}%
\ifdim\wd\matricesbox>\halfwidth\myboxwidth=\hsize\else\myboxwidth=\halfwidth\fi
\vbox{%
\ifdim\myboxwidth=\hsize
\setbox\onelinebox=\hbox{%
\vbox{\hbox{%
$\Pi_{4,38}=A_{2,II,\onebar}$ spans $L_{1.9}$%
}\hbox{%
$\slashinfty2\slashtwo2\rtimes D_{2}$ (shared)%
}%
}%
\hfill\copy\matricesbox
}%
\ifdim\wd\onelinebox>\myboxwidth
\hbox to \myboxwidth{%
$\Pi_{4,38}=A_{2,II,\onebar}$ spans $L_{1.9}$%
\hfil
$\slashinfty2\slashtwo2\rtimes D_{2}$ (shared)%
}%
\box\matricesbox
\else
\hbox to \myboxwidth{%
\unhbox\onelinebox
}%
\fi
\else
\hbox to \myboxwidth{%
$\Pi_{4,38}=A_{2,II,\onebar}$ spans $L_{1.9}$%
\hfil}%
\hbox to \myboxwidth{%
$\slashinfty2\slashtwo2\rtimes D_{2}$ (shared)%
\hfil}%
\box\matricesbox
\fi
}%
\hfill\discretionary{}{}{}%
\setbox\matricesbox=\hbox{%
{$\left[\!\llap{\phantom{%
\begingroup \smaller\smaller\smaller\begin{tabular}{@{}c@{}}%
\phantom{0}\\\phantom{0}\\\phantom{0}
\end{tabular}\endgroup%
}}\right.$}%
\begingroup \smaller\smaller\smaller\begin{tabular}{@{}c@{}}%
-1/8\\\phantom{0}\\\phantom{0}
\end{tabular}\endgroup%
\kern3pt%
\begingroup \smaller\smaller\smaller\begin{tabular}{@{}c@{}}%
\phantom{0}\\1/8\\\phantom{0}
\end{tabular}\endgroup%
\kern3pt%
\begingroup \smaller\smaller\smaller\begin{tabular}{@{}c@{}}%
\phantom{0}\\\phantom{0}\\1
\end{tabular}\endgroup%
{$\left.\llap{\phantom{%
\begingroup \smaller\smaller\smaller\begin{tabular}{@{}c@{}}%
\phantom{0}\\\phantom{0}\\\phantom{0}
\end{tabular}\endgroup%
}}\!\right]$}%
{$\left[\!\llap{\phantom{%
\begingroup \smaller\smaller\smaller\begin{tabular}{@{}c@{}}%
0\\0\\0
\end{tabular}\endgroup%
}}\right.$}%
\begingroup \smaller\smaller\smaller\begin{tabular}{@{}c@{}}%
1\\-1\\1
\end{tabular}\endgroup%
\kern3pt%
\begingroup \smaller\smaller\smaller\begin{tabular}{@{}c@{}}%
16\\16\\4
\end{tabular}\endgroup%
{$\left.\llap{\phantom{%
\begingroup \smaller\smaller\smaller\begin{tabular}{@{}c@{}}%
0\\0\\0
\end{tabular}\endgroup%
}}\!\right]$}%
}%
\ifdim\wd\matricesbox>\halfwidth\myboxwidth=\hsize\else\myboxwidth=\halfwidth\fi
\vbox{%
\ifdim\myboxwidth=\hsize
\setbox\onelinebox=\hbox{%
\vbox{\hbox{%
$\Pi_{4,39}=A_{4,II,\onebar}=\hbox{GN}_{15}$ spans $L_{140.4}$%
}\hbox{%
$\slashinfty2\slashinfty2\rtimes D_{2}$ (shared)%
}%
}%
\hfill\copy\matricesbox
}%
\ifdim\wd\onelinebox>\myboxwidth
\hbox to \myboxwidth{%
$\Pi_{4,39}=A_{4,II,\onebar}=\hbox{GN}_{15}$ spans $L_{140.4}$%
\hfil
$\slashinfty2\slashinfty2\rtimes D_{2}$ (shared)%
}%
\box\matricesbox
\else
\hbox to \myboxwidth{%
\unhbox\onelinebox
}%
\fi
\else
\hbox to \myboxwidth{%
$\Pi_{4,39}=A_{4,II,\onebar}=\hbox{GN}_{15}$ spans $L_{140.4}$%
\hfil}%
\hbox to \myboxwidth{%
$\slashinfty2\slashinfty2\rtimes D_{2}$ (shared)%
\hfil}%
\box\matricesbox
\fi
}%
\hfill\discretionary{}{}{}%
\setbox\matricesbox=\hbox{%
{$\left[\!\llap{\phantom{%
\begingroup \smaller\smaller\smaller\begin{tabular}{@{}c@{}}%
\phantom{0}\\\phantom{0}\\\phantom{0}
\end{tabular}\endgroup%
}}\right.$}%
\begingroup \smaller\smaller\smaller\begin{tabular}{@{}c@{}}%
-1/5\\\phantom{0}\\\phantom{0}
\end{tabular}\endgroup%
\kern3pt%
\begingroup \smaller\smaller\smaller\begin{tabular}{@{}c@{}}%
\phantom{0}\\1/5\\\phantom{0}
\end{tabular}\endgroup%
\kern3pt%
\begingroup \smaller\smaller\smaller\begin{tabular}{@{}c@{}}%
\phantom{0}\\\phantom{0}\\1
\end{tabular}\endgroup%
{$\left.\llap{\phantom{%
\begingroup \smaller\smaller\smaller\begin{tabular}{@{}c@{}}%
\phantom{0}\\\phantom{0}\\\phantom{0}
\end{tabular}\endgroup%
}}\!\right]$}%
{$\left[\!\llap{\phantom{%
\begingroup \smaller\smaller\smaller\begin{tabular}{@{}c@{}}%
0\\0\\0
\end{tabular}\endgroup%
}}\right.$}%
\begingroup \smaller\smaller\smaller\begin{tabular}{@{}c@{}}%
1\\1\\-1
\end{tabular}\endgroup%
\kern3pt%
\begingroup \smaller\smaller\smaller\begin{tabular}{@{}c@{}}%
2\\-3\\-1
\end{tabular}\endgroup%
{$\left.\llap{\phantom{%
\begingroup \smaller\smaller\smaller\begin{tabular}{@{}c@{}}%
0\\0\\0
\end{tabular}\endgroup%
}}\!\right]$}%
}%
\ifdim\wd\matricesbox>\halfwidth\myboxwidth=\hsize\else\myboxwidth=\halfwidth\fi
\vbox{%
\ifdim\myboxwidth=\hsize
\setbox\onelinebox=\hbox{%
\vbox{\hbox{%
$\Pi_{4,40}$ spans $L_{1.7}$%
}\hbox{%
$\slashinfty2\slashtwo2\rtimes D_{2}$ (shared)%
}%
}%
\hfill\copy\matricesbox
}%
\ifdim\wd\onelinebox>\myboxwidth
\hbox to \myboxwidth{%
$\Pi_{4,40}$ spans $L_{1.7}$%
\hfil
$\slashinfty2\slashtwo2\rtimes D_{2}$ (shared)%
}%
\box\matricesbox
\else
\hbox to \myboxwidth{%
\unhbox\onelinebox
}%
\fi
\else
\hbox to \myboxwidth{%
$\Pi_{4,40}$ spans $L_{1.7}$%
\hfil}%
\hbox to \myboxwidth{%
$\slashinfty2\slashtwo2\rtimes D_{2}$ (shared)%
\hfil}%
\box\matricesbox
\fi
}%
\hfill\discretionary{}{}{}%
\setbox\matricesbox=\hbox{%
{$\left[\!\llap{\phantom{%
\begingroup \smaller\smaller\smaller\begin{tabular}{@{}c@{}}%
\phantom{0}\\\phantom{0}\\\phantom{0}
\end{tabular}\endgroup%
}}\right.$}%
\begingroup \smaller\smaller\smaller\begin{tabular}{@{}c@{}}%
-1/4\\\phantom{0}\\\phantom{0}
\end{tabular}\endgroup%
\kern3pt%
\begingroup \smaller\smaller\smaller\begin{tabular}{@{}c@{}}%
\phantom{0}\\1/4\\\phantom{0}
\end{tabular}\endgroup%
\kern3pt%
\begingroup \smaller\smaller\smaller\begin{tabular}{@{}c@{}}%
\phantom{0}\\\phantom{0}\\1
\end{tabular}\endgroup%
{$\left.\llap{\phantom{%
\begingroup \smaller\smaller\smaller\begin{tabular}{@{}c@{}}%
\phantom{0}\\\phantom{0}\\\phantom{0}
\end{tabular}\endgroup%
}}\!\right]$}%
{$\left[\!\llap{\phantom{%
\begingroup \smaller\smaller\smaller\begin{tabular}{@{}c@{}}%
0\\0\\0
\end{tabular}\endgroup%
}}\right.$}%
\begingroup \smaller\smaller\smaller\begin{tabular}{@{}c@{}}%
1\\-1\\1
\end{tabular}\endgroup%
\kern3pt%
\begingroup \smaller\smaller\smaller\begin{tabular}{@{}c@{}}%
4\\4\\2
\end{tabular}\endgroup%
{$\left.\llap{\phantom{%
\begingroup \smaller\smaller\smaller\begin{tabular}{@{}c@{}}%
0\\0\\0
\end{tabular}\endgroup%
}}\!\right]$}%
}%
\ifdim\wd\matricesbox>\halfwidth\myboxwidth=\hsize\else\myboxwidth=\halfwidth\fi
\vbox{%
\ifdim\myboxwidth=\hsize
\setbox\onelinebox=\hbox{%
\vbox{\hbox{%
$\Pi_{4,41}=\hbox{GN}_{23}$ spans $L_{140.3}$%
}\hbox{%
$\slashinfty2\slashinfty2\rtimes D_{2}$ (shared)%
}%
}%
\hfill\copy\matricesbox
}%
\ifdim\wd\onelinebox>\myboxwidth
\hbox to \myboxwidth{%
$\Pi_{4,41}=\hbox{GN}_{23}$ spans $L_{140.3}$%
\hfil
$\slashinfty2\slashinfty2\rtimes D_{2}$ (shared)%
}%
\box\matricesbox
\else
\hbox to \myboxwidth{%
\unhbox\onelinebox
}%
\fi
\else
\hbox to \myboxwidth{%
$\Pi_{4,41}=\hbox{GN}_{23}$ spans $L_{140.3}$%
\hfil}%
\hbox to \myboxwidth{%
$\slashinfty2\slashinfty2\rtimes D_{2}$ (shared)%
\hfil}%
\box\matricesbox
\fi
}%
\hfill\discretionary{}{}{}%
\setbox\matricesbox=\hbox{%
{$\left[\!\llap{\phantom{%
\begingroup \smaller\smaller\smaller\begin{tabular}{@{}c@{}}%
\phantom{0}\\\phantom{0}\\\phantom{0}
\end{tabular}\endgroup%
}}\right.$}%
\begingroup \smaller\smaller\smaller\begin{tabular}{@{}c@{}}%
-1/6\\\phantom{0}\\\phantom{0}
\end{tabular}\endgroup%
\kern3pt%
\begingroup \smaller\smaller\smaller\begin{tabular}{@{}c@{}}%
\phantom{0}\\2/3\\\phantom{0}
\end{tabular}\endgroup%
\kern3pt%
\begingroup \smaller\smaller\smaller\begin{tabular}{@{}c@{}}%
\phantom{0}\\\phantom{0}\\1/2
\end{tabular}\endgroup%
{$\left.\llap{\phantom{%
\begingroup \smaller\smaller\smaller\begin{tabular}{@{}c@{}}%
\phantom{0}\\\phantom{0}\\\phantom{0}
\end{tabular}\endgroup%
}}\!\right]$}%
{$\left[\!\llap{\phantom{%
\begingroup \smaller\smaller\smaller\begin{tabular}{@{}c@{}}%
0\\0\\0
\end{tabular}\endgroup%
}}\right.$}%
\begingroup \smaller\smaller\smaller\begin{tabular}{@{}c@{}}%
2\\1\\2
\end{tabular}\endgroup%
\kern3pt%
\begingroup \smaller\smaller\smaller\begin{tabular}{@{}c@{}}%
1\\-1\\1
\end{tabular}\endgroup%
{$\left.\llap{\phantom{%
\begingroup \smaller\smaller\smaller\begin{tabular}{@{}c@{}}%
0\\0\\0
\end{tabular}\endgroup%
}}\!\right]$}%
}%
\ifdim\wd\matricesbox>\halfwidth\myboxwidth=\hsize\else\myboxwidth=\halfwidth\fi
\vbox{%
\ifdim\myboxwidth=\hsize
\setbox\onelinebox=\hbox{%
\vbox{\hbox{%
$\Pi_{4,42}$ spans $L_{1.5}$%
}\hbox{%
$\slashinfty2\slashtwo2\rtimes D_{2}$ (shared)%
}%
}%
\hfill\copy\matricesbox
}%
\ifdim\wd\onelinebox>\myboxwidth
\hbox to \myboxwidth{%
$\Pi_{4,42}$ spans $L_{1.5}$%
\hfil
$\slashinfty2\slashtwo2\rtimes D_{2}$ (shared)%
}%
\box\matricesbox
\else
\hbox to \myboxwidth{%
\unhbox\onelinebox
}%
\fi
\else
\hbox to \myboxwidth{%
$\Pi_{4,42}$ spans $L_{1.5}$%
\hfil}%
\hbox to \myboxwidth{%
$\slashinfty2\slashtwo2\rtimes D_{2}$ (shared)%
\hfil}%
\box\matricesbox
\fi
}%
\hfill\discretionary{}{}{}%
\setbox\matricesbox=\hbox{%
{$\left[\!\llap{\phantom{%
\begingroup \smaller\smaller\smaller\begin{tabular}{@{}c@{}}%
\phantom{0}\\\phantom{0}\\\phantom{0}
\end{tabular}\endgroup%
}}\right.$}%
\begingroup \smaller\smaller\smaller\begin{tabular}{@{}c@{}}%
-1/16\\\phantom{0}\\\phantom{0}
\end{tabular}\endgroup%
\kern3pt%
\begingroup \smaller\smaller\smaller\begin{tabular}{@{}c@{}}%
\phantom{0}\\1/16\\\phantom{0}
\end{tabular}\endgroup%
\kern3pt%
\begingroup \smaller\smaller\smaller\begin{tabular}{@{}c@{}}%
\phantom{0}\\\phantom{0}\\1/2
\end{tabular}\endgroup%
{$\left.\llap{\phantom{%
\begingroup \smaller\smaller\smaller\begin{tabular}{@{}c@{}}%
\phantom{0}\\\phantom{0}\\\phantom{0}
\end{tabular}\endgroup%
}}\!\right]$}%
{$\left[\!\llap{\phantom{%
\begingroup \smaller\smaller\smaller\begin{tabular}{@{}c@{}}%
0\\0\\0
\end{tabular}\endgroup%
}}\right.$}%
\begingroup \smaller\smaller\smaller\begin{tabular}{@{}c@{}}%
8\\-8\\4
\end{tabular}\endgroup%
\kern3pt%
\begingroup \smaller\smaller\smaller\begin{tabular}{@{}c@{}}%
1\\3\\1
\end{tabular}\endgroup%
{$\left.\llap{\phantom{%
\begingroup \smaller\smaller\smaller\begin{tabular}{@{}c@{}}%
0\\0\\0
\end{tabular}\endgroup%
}}\!\right]$}%
}%
\ifdim\wd\matricesbox>\halfwidth\myboxwidth=\hsize\else\myboxwidth=\halfwidth\fi
\vbox{%
\ifdim\myboxwidth=\hsize
\setbox\onelinebox=\hbox{%
\vbox{\hbox{%
$\Pi_{4,43}$ spans $L_{1.3}$%
}\hbox{%
$\slashinfty2\slashtwo2\rtimes D_{2}$ (shared)%
}%
}%
\hfill\copy\matricesbox
}%
\ifdim\wd\onelinebox>\myboxwidth
\hbox to \myboxwidth{%
$\Pi_{4,43}$ spans $L_{1.3}$%
\hfil
$\slashinfty2\slashtwo2\rtimes D_{2}$ (shared)%
}%
\box\matricesbox
\else
\hbox to \myboxwidth{%
\unhbox\onelinebox
}%
\fi
\else
\hbox to \myboxwidth{%
$\Pi_{4,43}$ spans $L_{1.3}$%
\hfil}%
\hbox to \myboxwidth{%
$\slashinfty2\slashtwo2\rtimes D_{2}$ (shared)%
\hfil}%
\box\matricesbox
\fi
}%
\hfill\discretionary{}{}{}%
\setbox\matricesbox=\hbox{%
{$\left[\!\llap{\phantom{%
\begingroup \smaller\smaller\smaller\begin{tabular}{@{}c@{}}%
\phantom{0}\\\phantom{0}\\\phantom{0}
\end{tabular}\endgroup%
}}\right.$}%
\begingroup \smaller\smaller\smaller\begin{tabular}{@{}c@{}}%
-1/8\\\phantom{0}\\\phantom{0}
\end{tabular}\endgroup%
\kern3pt%
\begingroup \smaller\smaller\smaller\begin{tabular}{@{}c@{}}%
\phantom{0}\\5/2\\-1
\end{tabular}\endgroup%
\kern3pt%
\begingroup \smaller\smaller\smaller\begin{tabular}{@{}c@{}}%
\phantom{0}\\-1\\6
\end{tabular}\endgroup%
{$\left.\llap{\phantom{%
\begingroup \smaller\smaller\smaller\begin{tabular}{@{}c@{}}%
\phantom{0}\\\phantom{0}\\\phantom{0}
\end{tabular}\endgroup%
}}\!\right]$}%
{$\left[\!\llap{\phantom{%
\begingroup \smaller\smaller\smaller\begin{tabular}{@{}c@{}}%
0\\0\\0
\end{tabular}\endgroup%
}}\right.$}%
\begingroup \smaller\smaller\smaller\begin{tabular}{@{}c@{}}%
2\\1\\0
\end{tabular}\endgroup%
\kern3pt%
\begingroup \smaller\smaller\smaller\begin{tabular}{@{}c@{}}%
4\\0\\1
\end{tabular}\endgroup%
{$\left.\llap{\phantom{%
\begingroup \smaller\smaller\smaller\begin{tabular}{@{}c@{}}%
0\\0\\0
\end{tabular}\endgroup%
}}\!\right]$}%
}%
\ifdim\wd\matricesbox>\halfwidth\myboxwidth=\hsize\else\myboxwidth=\halfwidth\fi
\vbox{%
\ifdim\myboxwidth=\hsize
\setbox\onelinebox=\hbox{%
\vbox{\hbox{%
$\Pi_{4,44}$ spans $L_{8.1}$%
}\hbox{%
$4242\rtimes C_{2}$%
}%
}%
\hfill\copy\matricesbox
}%
\ifdim\wd\onelinebox>\myboxwidth
\hbox to \myboxwidth{%
$\Pi_{4,44}$ spans $L_{8.1}$%
\hfil
$4242\rtimes C_{2}$%
}%
\box\matricesbox
\else
\hbox to \myboxwidth{%
\unhbox\onelinebox
}%
\fi
\else
\hbox to \myboxwidth{%
$\Pi_{4,44}$ spans $L_{8.1}$%
\hfil}%
\hbox to \myboxwidth{%
$4242\rtimes C_{2}$%
\hfil}%
\box\matricesbox
\fi
}%
\hfill\discretionary{}{}{}%
\setbox\matricesbox=\hbox{%
{$\left[\!\llap{\phantom{%
\begingroup \smaller\smaller\smaller\begin{tabular}{@{}c@{}}%
\phantom{0}\\\phantom{0}\\\phantom{0}
\end{tabular}\endgroup%
}}\right.$}%
\begingroup \smaller\smaller\smaller\begin{tabular}{@{}c@{}}%
-1/8\\\phantom{0}\\\phantom{0}
\end{tabular}\endgroup%
\kern3pt%
\begingroup \smaller\smaller\smaller\begin{tabular}{@{}c@{}}%
\phantom{0}\\5/2\\-1
\end{tabular}\endgroup%
\kern3pt%
\begingroup \smaller\smaller\smaller\begin{tabular}{@{}c@{}}%
\phantom{0}\\-1\\10
\end{tabular}\endgroup%
{$\left.\llap{\phantom{%
\begingroup \smaller\smaller\smaller\begin{tabular}{@{}c@{}}%
\phantom{0}\\\phantom{0}\\\phantom{0}
\end{tabular}\endgroup%
}}\!\right]$}%
{$\left[\!\llap{\phantom{%
\begingroup \smaller\smaller\smaller\begin{tabular}{@{}c@{}}%
0\\0\\0
\end{tabular}\endgroup%
}}\right.$}%
\begingroup \smaller\smaller\smaller\begin{tabular}{@{}c@{}}%
2\\1\\0
\end{tabular}\endgroup%
\kern3pt%
\begingroup \smaller\smaller\smaller\begin{tabular}{@{}c@{}}%
6\\-1\\-1
\end{tabular}\endgroup%
{$\left.\llap{\phantom{%
\begingroup \smaller\smaller\smaller\begin{tabular}{@{}c@{}}%
0\\0\\0
\end{tabular}\endgroup%
}}\!\right]$}%
}%
\ifdim\wd\matricesbox>\halfwidth\myboxwidth=\hsize\else\myboxwidth=\halfwidth\fi
\vbox{%
\ifdim\myboxwidth=\hsize
\setbox\onelinebox=\hbox{%
\vbox{\hbox{%
$\Pi_{4,45}$ spans $L_{168.2}$%
}\hbox{%
$6262\rtimes C_{2}$%
}%
}%
\hfill\copy\matricesbox
}%
\ifdim\wd\onelinebox>\myboxwidth
\hbox to \myboxwidth{%
$\Pi_{4,45}$ spans $L_{168.2}$%
\hfil
$6262\rtimes C_{2}$%
}%
\box\matricesbox
\else
\hbox to \myboxwidth{%
\unhbox\onelinebox
}%
\fi
\else
\hbox to \myboxwidth{%
$\Pi_{4,45}$ spans $L_{168.2}$%
\hfil}%
\hbox to \myboxwidth{%
$6262\rtimes C_{2}$%
\hfil}%
\box\matricesbox
\fi
}%
\hfill\discretionary{}{}{}%
\setbox\matricesbox=\hbox{%
{$\left[\!\llap{\phantom{%
\begingroup \smaller\smaller\smaller\begin{tabular}{@{}c@{}}%
\phantom{0}\\\phantom{0}\\\phantom{0}
\end{tabular}\endgroup%
}}\right.$}%
\begingroup \smaller\smaller\smaller\begin{tabular}{@{}c@{}}%
-1/8\\\phantom{0}\\\phantom{0}
\end{tabular}\endgroup%
\kern3pt%
\begingroup \smaller\smaller\smaller\begin{tabular}{@{}c@{}}%
\phantom{0}\\5/2\\-1/2
\end{tabular}\endgroup%
\kern3pt%
\begingroup \smaller\smaller\smaller\begin{tabular}{@{}c@{}}%
\phantom{0}\\-1/2\\29/2
\end{tabular}\endgroup%
{$\left.\llap{\phantom{%
\begingroup \smaller\smaller\smaller\begin{tabular}{@{}c@{}}%
\phantom{0}\\\phantom{0}\\\phantom{0}
\end{tabular}\endgroup%
}}\!\right]$}%
{$\left[\!\llap{\phantom{%
\begingroup \smaller\smaller\smaller\begin{tabular}{@{}c@{}}%
0\\0\\0
\end{tabular}\endgroup%
}}\right.$}%
\begingroup \smaller\smaller\smaller\begin{tabular}{@{}c@{}}%
2\\1\\0
\end{tabular}\endgroup%
\kern3pt%
\begingroup \smaller\smaller\smaller\begin{tabular}{@{}c@{}}%
8\\-1\\-1
\end{tabular}\endgroup%
{$\left.\llap{\phantom{%
\begingroup \smaller\smaller\smaller\begin{tabular}{@{}c@{}}%
0\\0\\0
\end{tabular}\endgroup%
}}\!\right]$}%
}%
\ifdim\wd\matricesbox>\halfwidth\myboxwidth=\hsize\else\myboxwidth=\halfwidth\fi
\vbox{%
\ifdim\myboxwidth=\hsize
\setbox\onelinebox=\hbox{%
\vbox{\hbox{%
$\Pi_{4,46}=\hbox{GN}_{22}$ spans $L_{148.2}$%
}\hbox{%
$\infty2\infty2\rtimes C_{2}$ (shared)%
}%
}%
\hfill\copy\matricesbox
}%
\ifdim\wd\onelinebox>\myboxwidth
\hbox to \myboxwidth{%
$\Pi_{4,46}=\hbox{GN}_{22}$ spans $L_{148.2}$%
\hfil
$\infty2\infty2\rtimes C_{2}$ (shared)%
}%
\box\matricesbox
\else
\hbox to \myboxwidth{%
\unhbox\onelinebox
}%
\fi
\else
\hbox to \myboxwidth{%
$\Pi_{4,46}=\hbox{GN}_{22}$ spans $L_{148.2}$%
\hfil}%
\hbox to \myboxwidth{%
$\infty2\infty2\rtimes C_{2}$ (shared)%
\hfil}%
\box\matricesbox
\fi
}%
\hfill\discretionary{}{}{}%
\setbox\matricesbox=\hbox{%
{$\left[\!\llap{\phantom{%
\begingroup \smaller\smaller\smaller\begin{tabular}{@{}c@{}}%
\phantom{0}\\\phantom{0}\\\phantom{0}
\end{tabular}\endgroup%
}}\right.$}%
\begingroup \smaller\smaller\smaller\begin{tabular}{@{}c@{}}%
-1/28\\\phantom{0}\\\phantom{0}
\end{tabular}\endgroup%
\kern3pt%
\begingroup \smaller\smaller\smaller\begin{tabular}{@{}c@{}}%
\phantom{0}\\15/7\\-6/7
\end{tabular}\endgroup%
\kern3pt%
\begingroup \smaller\smaller\smaller\begin{tabular}{@{}c@{}}%
\phantom{0}\\-6/7\\15/7
\end{tabular}\endgroup%
{$\left.\llap{\phantom{%
\begingroup \smaller\smaller\smaller\begin{tabular}{@{}c@{}}%
\phantom{0}\\\phantom{0}\\\phantom{0}
\end{tabular}\endgroup%
}}\!\right]$}%
{$\left[\!\llap{\phantom{%
\begingroup \smaller\smaller\smaller\begin{tabular}{@{}c@{}}%
0\\0\\0
\end{tabular}\endgroup%
}}\right.$}%
\begingroup \smaller\smaller\smaller\begin{tabular}{@{}c@{}}%
18\\-1\\-4
\end{tabular}\endgroup%
\kern3pt%
\begingroup \smaller\smaller\smaller\begin{tabular}{@{}c@{}}%
18\\-4\\-1
\end{tabular}\endgroup%
\kern3pt%
\begingroup \smaller\smaller\smaller\begin{tabular}{@{}c@{}}%
6\\1\\2
\end{tabular}\endgroup%
\kern3pt%
\begingroup \smaller\smaller\smaller\begin{tabular}{@{}c@{}}%
2\\1\\0
\end{tabular}\endgroup%
{$\left.\llap{\phantom{%
\begingroup \smaller\smaller\smaller\begin{tabular}{@{}c@{}}%
0\\0\\0
\end{tabular}\endgroup%
}}\!\right]$}%
}%
\ifdim\wd\matricesbox>\halfwidth\myboxwidth=\hsize\else\myboxwidth=\halfwidth\fi
\vbox{%
\ifdim\myboxwidth=\hsize
\setbox\onelinebox=\hbox{%
\vbox{\hbox{%
$\Pi_{4,47}$ spans $L_{155.1}$%
}\hbox{%
$3622$ (shared)%
}%
}%
\hfill\copy\matricesbox
}%
\ifdim\wd\onelinebox>\myboxwidth
\hbox to \myboxwidth{%
$\Pi_{4,47}$ spans $L_{155.1}$%
\hfil
$3622$ (shared)%
}%
\box\matricesbox
\else
\hbox to \myboxwidth{%
\unhbox\onelinebox
}%
\fi
\else
\hbox to \myboxwidth{%
$\Pi_{4,47}$ spans $L_{155.1}$%
\hfil}%
\hbox to \myboxwidth{%
$3622$ (shared)%
\hfil}%
\box\matricesbox
\fi
}%
\hfill\discretionary{}{}{}%
\setbox\matricesbox=\hbox{%
{$\left[\!\llap{\phantom{%
\begingroup \smaller\smaller\smaller\begin{tabular}{@{}c@{}}%
\phantom{0}\\\phantom{0}\\\phantom{0}
\end{tabular}\endgroup%
}}\right.$}%
\begingroup \smaller\smaller\smaller\begin{tabular}{@{}c@{}}%
-3/25\\\phantom{0}\\\phantom{0}
\end{tabular}\endgroup%
\kern3pt%
\begingroup \smaller\smaller\smaller\begin{tabular}{@{}c@{}}%
\phantom{0}\\12/25\\-6/25
\end{tabular}\endgroup%
\kern3pt%
\begingroup \smaller\smaller\smaller\begin{tabular}{@{}c@{}}%
\phantom{0}\\-6/25\\28/25
\end{tabular}\endgroup%
{$\left.\llap{\phantom{%
\begingroup \smaller\smaller\smaller\begin{tabular}{@{}c@{}}%
\phantom{0}\\\phantom{0}\\\phantom{0}
\end{tabular}\endgroup%
}}\!\right]$}%
{$\left[\!\llap{\phantom{%
\begingroup \smaller\smaller\smaller\begin{tabular}{@{}c@{}}%
0\\0\\0
\end{tabular}\endgroup%
}}\right.$}%
\begingroup \smaller\smaller\smaller\begin{tabular}{@{}c@{}}%
12\\5\\-3
\end{tabular}\endgroup%
\kern3pt%
\begingroup \smaller\smaller\smaller\begin{tabular}{@{}c@{}}%
12\\8\\3
\end{tabular}\endgroup%
\kern3pt%
\begingroup \smaller\smaller\smaller\begin{tabular}{@{}c@{}}%
4\\-1\\2
\end{tabular}\endgroup%
\kern3pt%
\begingroup \smaller\smaller\smaller\begin{tabular}{@{}c@{}}%
1\\-1\\-1
\end{tabular}\endgroup%
{$\left.\llap{\phantom{%
\begingroup \smaller\smaller\smaller\begin{tabular}{@{}c@{}}%
0\\0\\0
\end{tabular}\endgroup%
}}\!\right]$}%
}%
\ifdim\wd\matricesbox>\halfwidth\myboxwidth=\hsize\else\myboxwidth=\halfwidth\fi
\vbox{%
\ifdim\myboxwidth=\hsize
\setbox\onelinebox=\hbox{%
\vbox{\hbox{%
$\Pi_{4,48}$ spans $L_{221.3}$%
}\hbox{%
$36\infty2$ (shared)%
}%
}%
\hfill\copy\matricesbox
}%
\ifdim\wd\onelinebox>\myboxwidth
\hbox to \myboxwidth{%
$\Pi_{4,48}$ spans $L_{221.3}$%
\hfil
$36\infty2$ (shared)%
}%
\box\matricesbox
\else
\hbox to \myboxwidth{%
\unhbox\onelinebox
}%
\fi
\else
\hbox to \myboxwidth{%
$\Pi_{4,48}$ spans $L_{221.3}$%
\hfil}%
\hbox to \myboxwidth{%
$36\infty2$ (shared)%
\hfil}%
\box\matricesbox
\fi
}%
\hfill\discretionary{}{}{}%
\setbox\matricesbox=\hbox{%
{$\left[\!\llap{\phantom{%
\begingroup \smaller\smaller\smaller\begin{tabular}{@{}c@{}}%
\phantom{0}\\\phantom{0}\\\phantom{0}
\end{tabular}\endgroup%
}}\right.$}%
\begingroup \smaller\smaller\smaller\begin{tabular}{@{}c@{}}%
-1/8\\\phantom{0}\\\phantom{0}
\end{tabular}\endgroup%
\kern3pt%
\begingroup \smaller\smaller\smaller\begin{tabular}{@{}c@{}}%
\phantom{0}\\5/2\\-1/2
\end{tabular}\endgroup%
\kern3pt%
\begingroup \smaller\smaller\smaller\begin{tabular}{@{}c@{}}%
\phantom{0}\\-1/2\\5/2
\end{tabular}\endgroup%
{$\left.\llap{\phantom{%
\begingroup \smaller\smaller\smaller\begin{tabular}{@{}c@{}}%
\phantom{0}\\\phantom{0}\\\phantom{0}
\end{tabular}\endgroup%
}}\!\right]$}%
{$\left[\!\llap{\phantom{%
\begingroup \smaller\smaller\smaller\begin{tabular}{@{}c@{}}%
0\\0\\0
\end{tabular}\endgroup%
}}\right.$}%
\begingroup \smaller\smaller\smaller\begin{tabular}{@{}c@{}}%
2\\0\\1
\end{tabular}\endgroup%
\kern3pt%
\begingroup \smaller\smaller\smaller\begin{tabular}{@{}c@{}}%
2\\1\\0
\end{tabular}\endgroup%
\kern3pt%
\begingroup \smaller\smaller\smaller\begin{tabular}{@{}c@{}}%
6\\-1\\-2
\end{tabular}\endgroup%
\kern3pt%
\begingroup \smaller\smaller\smaller\begin{tabular}{@{}c@{}}%
2\\-1\\0
\end{tabular}\endgroup%
{$\left.\llap{\phantom{%
\begingroup \smaller\smaller\smaller\begin{tabular}{@{}c@{}}%
0\\0\\0
\end{tabular}\endgroup%
}}\!\right]$}%
}%
\ifdim\wd\matricesbox>\halfwidth\myboxwidth=\hsize\else\myboxwidth=\halfwidth\fi
\vbox{%
\ifdim\myboxwidth=\hsize
\setbox\onelinebox=\hbox{%
\vbox{\hbox{%
$\Pi_{4,49}$ spans $L_{3.1}$%
}\hbox{%
$3622$ (shared)%
}%
}%
\hfill\copy\matricesbox
}%
\ifdim\wd\onelinebox>\myboxwidth
\hbox to \myboxwidth{%
$\Pi_{4,49}$ spans $L_{3.1}$%
\hfil
$3622$ (shared)%
}%
\box\matricesbox
\else
\hbox to \myboxwidth{%
\unhbox\onelinebox
}%
\fi
\else
\hbox to \myboxwidth{%
$\Pi_{4,49}$ spans $L_{3.1}$%
\hfil}%
\hbox to \myboxwidth{%
$3622$ (shared)%
\hfil}%
\box\matricesbox
\fi
}%
\hfill\discretionary{}{}{}%
\setbox\matricesbox=\hbox{%
{$\left[\!\llap{\phantom{%
\begingroup \smaller\smaller\smaller\begin{tabular}{@{}c@{}}%
\phantom{0}\\\phantom{0}\\\phantom{0}
\end{tabular}\endgroup%
}}\right.$}%
\begingroup \smaller\smaller\smaller\begin{tabular}{@{}c@{}}%
-1/9\\\phantom{0}\\\phantom{0}
\end{tabular}\endgroup%
\kern3pt%
\begingroup \smaller\smaller\smaller\begin{tabular}{@{}c@{}}%
\phantom{0}\\4/9\\-2/9
\end{tabular}\endgroup%
\kern3pt%
\begingroup \smaller\smaller\smaller\begin{tabular}{@{}c@{}}%
\phantom{0}\\-2/9\\28/9
\end{tabular}\endgroup%
{$\left.\llap{\phantom{%
\begingroup \smaller\smaller\smaller\begin{tabular}{@{}c@{}}%
\phantom{0}\\\phantom{0}\\\phantom{0}
\end{tabular}\endgroup%
}}\!\right]$}%
{$\left[\!\llap{\phantom{%
\begingroup \smaller\smaller\smaller\begin{tabular}{@{}c@{}}%
0\\0\\0
\end{tabular}\endgroup%
}}\right.$}%
\begingroup \smaller\smaller\smaller\begin{tabular}{@{}c@{}}%
4\\-3\\-1
\end{tabular}\endgroup%
\kern3pt%
\begingroup \smaller\smaller\smaller\begin{tabular}{@{}c@{}}%
4\\-2\\1
\end{tabular}\endgroup%
\kern3pt%
\begingroup \smaller\smaller\smaller\begin{tabular}{@{}c@{}}%
12\\7\\2
\end{tabular}\endgroup%
\kern3pt%
\begingroup \smaller\smaller\smaller\begin{tabular}{@{}c@{}}%
3\\1\\-1
\end{tabular}\endgroup%
{$\left.\llap{\phantom{%
\begingroup \smaller\smaller\smaller\begin{tabular}{@{}c@{}}%
0\\0\\0
\end{tabular}\endgroup%
}}\!\right]$}%
}%
\ifdim\wd\matricesbox>\halfwidth\myboxwidth=\hsize\else\myboxwidth=\halfwidth\fi
\vbox{%
\ifdim\myboxwidth=\hsize
\setbox\onelinebox=\hbox{%
\vbox{\hbox{%
$\Pi_{4,50}$ spans $L_{7.11}$%
}\hbox{%
$36\infty2$ (shared)%
}%
}%
\hfill\copy\matricesbox
}%
\ifdim\wd\onelinebox>\myboxwidth
\hbox to \myboxwidth{%
$\Pi_{4,50}$ spans $L_{7.11}$%
\hfil
$36\infty2$ (shared)%
}%
\box\matricesbox
\else
\hbox to \myboxwidth{%
\unhbox\onelinebox
}%
\fi
\else
\hbox to \myboxwidth{%
$\Pi_{4,50}$ spans $L_{7.11}$%
\hfil}%
\hbox to \myboxwidth{%
$36\infty2$ (shared)%
\hfil}%
\box\matricesbox
\fi
}%
\hfill\discretionary{}{}{}%
\setbox\matricesbox=\hbox{%
{$\left[\!\llap{\phantom{%
\begingroup \smaller\smaller\smaller\begin{tabular}{@{}c@{}}%
\phantom{0}\\\phantom{0}\\\phantom{0}
\end{tabular}\endgroup%
}}\right.$}%
\begingroup \smaller\smaller\smaller\begin{tabular}{@{}c@{}}%
-1/5\\\phantom{0}\\\phantom{0}
\end{tabular}\endgroup%
\kern3pt%
\begingroup \smaller\smaller\smaller\begin{tabular}{@{}c@{}}%
\phantom{0}\\4/5\\-2/5
\end{tabular}\endgroup%
\kern3pt%
\begingroup \smaller\smaller\smaller\begin{tabular}{@{}c@{}}%
\phantom{0}\\-2/5\\6/5
\end{tabular}\endgroup%
{$\left.\llap{\phantom{%
\begingroup \smaller\smaller\smaller\begin{tabular}{@{}c@{}}%
\phantom{0}\\\phantom{0}\\\phantom{0}
\end{tabular}\endgroup%
}}\!\right]$}%
{$\left[\!\llap{\phantom{%
\begingroup \smaller\smaller\smaller\begin{tabular}{@{}c@{}}%
0\\0\\0
\end{tabular}\endgroup%
}}\right.$}%
\begingroup \smaller\smaller\smaller\begin{tabular}{@{}c@{}}%
2\\-1\\1
\end{tabular}\endgroup%
\kern3pt%
\begingroup \smaller\smaller\smaller\begin{tabular}{@{}c@{}}%
4\\3\\2
\end{tabular}\endgroup%
\kern3pt%
\begingroup \smaller\smaller\smaller\begin{tabular}{@{}c@{}}%
1\\0\\-1
\end{tabular}\endgroup%
\kern3pt%
\begingroup \smaller\smaller\smaller\begin{tabular}{@{}c@{}}%
2\\-2\\-1
\end{tabular}\endgroup%
{$\left.\llap{\phantom{%
\begingroup \smaller\smaller\smaller\begin{tabular}{@{}c@{}}%
0\\0\\0
\end{tabular}\endgroup%
}}\!\right]$}%
}%
\ifdim\wd\matricesbox>\halfwidth\myboxwidth=\hsize\else\myboxwidth=\halfwidth\fi
\vbox{%
\ifdim\myboxwidth=\hsize
\setbox\onelinebox=\hbox{%
\vbox{\hbox{%
$\Pi_{4,51}$ spans $L_{1.7}$%
}\hbox{%
$4\infty22$ (shared)%
}%
}%
\hfill\copy\matricesbox
}%
\ifdim\wd\onelinebox>\myboxwidth
\hbox to \myboxwidth{%
$\Pi_{4,51}$ spans $L_{1.7}$%
\hfil
$4\infty22$ (shared)%
}%
\box\matricesbox
\else
\hbox to \myboxwidth{%
\unhbox\onelinebox
}%
\fi
\else
\hbox to \myboxwidth{%
$\Pi_{4,51}$ spans $L_{1.7}$%
\hfil}%
\hbox to \myboxwidth{%
$4\infty22$ (shared)%
\hfil}%
\box\matricesbox
\fi
}%
\hfill\discretionary{}{}{}%
\setbox\matricesbox=\hbox{%
{$\left[\!\llap{\phantom{%
\begingroup \smaller\smaller\smaller\begin{tabular}{@{}c@{}}%
\phantom{0}\\\phantom{0}\\\phantom{0}
\end{tabular}\endgroup%
}}\right.$}%
\begingroup \smaller\smaller\smaller\begin{tabular}{@{}c@{}}%
-3/49\\\phantom{0}\\\phantom{0}
\end{tabular}\endgroup%
\kern3pt%
\begingroup \smaller\smaller\smaller\begin{tabular}{@{}c@{}}%
\phantom{0}\\20/49\\-4/49
\end{tabular}\endgroup%
\kern3pt%
\begingroup \smaller\smaller\smaller\begin{tabular}{@{}c@{}}%
\phantom{0}\\-4/49\\40/49
\end{tabular}\endgroup%
{$\left.\llap{\phantom{%
\begingroup \smaller\smaller\smaller\begin{tabular}{@{}c@{}}%
\phantom{0}\\\phantom{0}\\\phantom{0}
\end{tabular}\endgroup%
}}\!\right]$}%
{$\left[\!\llap{\phantom{%
\begingroup \smaller\smaller\smaller\begin{tabular}{@{}c@{}}%
0\\0\\0
\end{tabular}\endgroup%
}}\right.$}%
\begingroup \smaller\smaller\smaller\begin{tabular}{@{}c@{}}%
4\\3\\-1
\end{tabular}\endgroup%
\kern3pt%
\begingroup \smaller\smaller\smaller\begin{tabular}{@{}c@{}}%
8\\-4\\-3
\end{tabular}\endgroup%
\kern3pt%
\begingroup \smaller\smaller\smaller\begin{tabular}{@{}c@{}}%
16\\-6\\4
\end{tabular}\endgroup%
\kern3pt%
\begingroup \smaller\smaller\smaller\begin{tabular}{@{}c@{}}%
1\\1\\1
\end{tabular}\endgroup%
{$\left.\llap{\phantom{%
\begingroup \smaller\smaller\smaller\begin{tabular}{@{}c@{}}%
0\\0\\0
\end{tabular}\endgroup%
}}\!\right]$}%
}%
\ifdim\wd\matricesbox>\halfwidth\myboxwidth=\hsize\else\myboxwidth=\halfwidth\fi
\vbox{%
\ifdim\myboxwidth=\hsize
\setbox\onelinebox=\hbox{%
\vbox{\hbox{%
$\Pi_{4,52}$ spans $L_{123.9}$%
}\hbox{%
$4422$ (shared)%
}%
}%
\hfill\copy\matricesbox
}%
\ifdim\wd\onelinebox>\myboxwidth
\hbox to \myboxwidth{%
$\Pi_{4,52}$ spans $L_{123.9}$%
\hfil
$4422$ (shared)%
}%
\box\matricesbox
\else
\hbox to \myboxwidth{%
\unhbox\onelinebox
}%
\fi
\else
\hbox to \myboxwidth{%
$\Pi_{4,52}$ spans $L_{123.9}$%
\hfil}%
\hbox to \myboxwidth{%
$4422$ (shared)%
\hfil}%
\box\matricesbox
\fi
}%
\hfill\discretionary{}{}{}%
\setbox\matricesbox=\hbox{%
{$\left[\!\llap{\phantom{%
\begingroup \smaller\smaller\smaller\begin{tabular}{@{}c@{}}%
\phantom{0}\\\phantom{0}\\\phantom{0}
\end{tabular}\endgroup%
}}\right.$}%
\begingroup \smaller\smaller\smaller\begin{tabular}{@{}c@{}}%
-1/8\\\phantom{0}\\\phantom{0}
\end{tabular}\endgroup%
\kern3pt%
\begingroup \smaller\smaller\smaller\begin{tabular}{@{}c@{}}%
\phantom{0}\\9/8\\-3/8
\end{tabular}\endgroup%
\kern3pt%
\begingroup \smaller\smaller\smaller\begin{tabular}{@{}c@{}}%
\phantom{0}\\-3/8\\17/8
\end{tabular}\endgroup%
{$\left.\llap{\phantom{%
\begingroup \smaller\smaller\smaller\begin{tabular}{@{}c@{}}%
\phantom{0}\\\phantom{0}\\\phantom{0}
\end{tabular}\endgroup%
}}\!\right]$}%
{$\left[\!\llap{\phantom{%
\begingroup \smaller\smaller\smaller\begin{tabular}{@{}c@{}}%
0\\0\\0
\end{tabular}\endgroup%
}}\right.$}%
\begingroup \smaller\smaller\smaller\begin{tabular}{@{}c@{}}%
1\\-1\\0
\end{tabular}\endgroup%
\kern3pt%
\begingroup \smaller\smaller\smaller\begin{tabular}{@{}c@{}}%
2\\1\\1
\end{tabular}\endgroup%
\kern3pt%
\begingroup \smaller\smaller\smaller\begin{tabular}{@{}c@{}}%
8\\2\\-2
\end{tabular}\endgroup%
\kern3pt%
\begingroup \smaller\smaller\smaller\begin{tabular}{@{}c@{}}%
9\\-2\\-3
\end{tabular}\endgroup%
{$\left.\llap{\phantom{%
\begingroup \smaller\smaller\smaller\begin{tabular}{@{}c@{}}%
0\\0\\0
\end{tabular}\endgroup%
}}\!\right]$}%
}%
\ifdim\wd\matricesbox>\halfwidth\myboxwidth=\hsize\else\myboxwidth=\halfwidth\fi
\vbox{%
\ifdim\myboxwidth=\hsize
\setbox\onelinebox=\hbox{%
\vbox{\hbox{%
$\Pi_{4,53}$ spans $L_{142.3}$%
}\hbox{%
$4\infty22$ (shared)%
}%
}%
\hfill\copy\matricesbox
}%
\ifdim\wd\onelinebox>\myboxwidth
\hbox to \myboxwidth{%
$\Pi_{4,53}$ spans $L_{142.3}$%
\hfil
$4\infty22$ (shared)%
}%
\box\matricesbox
\else
\hbox to \myboxwidth{%
\unhbox\onelinebox
}%
\fi
\else
\hbox to \myboxwidth{%
$\Pi_{4,53}$ spans $L_{142.3}$%
\hfil}%
\hbox to \myboxwidth{%
$4\infty22$ (shared)%
\hfil}%
\box\matricesbox
\fi
}%
\hfill\discretionary{}{}{}%
\setbox\matricesbox=\hbox{%
{$\left[\!\llap{\phantom{%
\begingroup \smaller\smaller\smaller\begin{tabular}{@{}c@{}}%
\phantom{0}\\\phantom{0}\\\phantom{0}
\end{tabular}\endgroup%
}}\right.$}%
\begingroup \smaller\smaller\smaller\begin{tabular}{@{}c@{}}%
-1/9\\\phantom{0}\\\phantom{0}
\end{tabular}\endgroup%
\kern3pt%
\begingroup \smaller\smaller\smaller\begin{tabular}{@{}c@{}}%
\phantom{0}\\4/9\\-2/9
\end{tabular}\endgroup%
\kern3pt%
\begingroup \smaller\smaller\smaller\begin{tabular}{@{}c@{}}%
\phantom{0}\\-2/9\\28/9
\end{tabular}\endgroup%
{$\left.\llap{\phantom{%
\begingroup \smaller\smaller\smaller\begin{tabular}{@{}c@{}}%
\phantom{0}\\\phantom{0}\\\phantom{0}
\end{tabular}\endgroup%
}}\!\right]$}%
{$\left[\!\llap{\phantom{%
\begingroup \smaller\smaller\smaller\begin{tabular}{@{}c@{}}%
0\\0\\0
\end{tabular}\endgroup%
}}\right.$}%
\begingroup \smaller\smaller\smaller\begin{tabular}{@{}c@{}}%
4\\-2\\1
\end{tabular}\endgroup%
\kern3pt%
\begingroup \smaller\smaller\smaller\begin{tabular}{@{}c@{}}%
12\\7\\2
\end{tabular}\endgroup%
\kern3pt%
\begingroup \smaller\smaller\smaller\begin{tabular}{@{}c@{}}%
3\\1\\-1
\end{tabular}\endgroup%
\kern3pt%
\begingroup \smaller\smaller\smaller\begin{tabular}{@{}c@{}}%
12\\-8\\-1
\end{tabular}\endgroup%
{$\left.\llap{\phantom{%
\begingroup \smaller\smaller\smaller\begin{tabular}{@{}c@{}}%
0\\0\\0
\end{tabular}\endgroup%
}}\!\right]$}%
}%
\ifdim\wd\matricesbox>\halfwidth\myboxwidth=\hsize\else\myboxwidth=\halfwidth\fi
\vbox{%
\ifdim\myboxwidth=\hsize
\setbox\onelinebox=\hbox{%
\vbox{\hbox{%
$\Pi_{4,54}$ spans $L_{7.11}$%
}\hbox{%
$6\infty\infty2$%
}%
}%
\hfill\copy\matricesbox
}%
\ifdim\wd\onelinebox>\myboxwidth
\hbox to \myboxwidth{%
$\Pi_{4,54}$ spans $L_{7.11}$%
\hfil
$6\infty\infty2$%
}%
\box\matricesbox
\else
\hbox to \myboxwidth{%
\unhbox\onelinebox
}%
\fi
\else
\hbox to \myboxwidth{%
$\Pi_{4,54}$ spans $L_{7.11}$%
\hfil}%
\hbox to \myboxwidth{%
$6\infty\infty2$%
\hfil}%
\box\matricesbox
\fi
}%
\hfill\discretionary{}{}{}%
\setbox\matricesbox=\hbox{%
{$\left[\!\llap{\phantom{%
\begingroup \smaller\smaller\smaller\begin{tabular}{@{}c@{}}%
\phantom{0}\\\phantom{0}\\\phantom{0}
\end{tabular}\endgroup%
}}\right.$}%
\begingroup \smaller\smaller\smaller\begin{tabular}{@{}c@{}}%
-1/2\\\phantom{0}\\\phantom{0}
\end{tabular}\endgroup%
\kern3pt%
\begingroup \smaller\smaller\smaller\begin{tabular}{@{}c@{}}%
\phantom{0}\\3/2\\-1/2
\end{tabular}\endgroup%
\kern3pt%
\begingroup \smaller\smaller\smaller\begin{tabular}{@{}c@{}}%
\phantom{0}\\-1/2\\3/2
\end{tabular}\endgroup%
{$\left.\llap{\phantom{%
\begingroup \smaller\smaller\smaller\begin{tabular}{@{}c@{}}%
\phantom{0}\\\phantom{0}\\\phantom{0}
\end{tabular}\endgroup%
}}\!\right]$}%
{$\left[\!\llap{\phantom{%
\begingroup \smaller\smaller\smaller\begin{tabular}{@{}c@{}}%
0\\0\\0
\end{tabular}\endgroup%
}}\right.$}%
\begingroup \smaller\smaller\smaller\begin{tabular}{@{}c@{}}%
1\\-1\\0
\end{tabular}\endgroup%
\kern3pt%
\begingroup \smaller\smaller\smaller\begin{tabular}{@{}c@{}}%
4\\1\\3
\end{tabular}\endgroup%
\kern3pt%
\begingroup \smaller\smaller\smaller\begin{tabular}{@{}c@{}}%
1\\1\\0
\end{tabular}\endgroup%
\kern3pt%
\begingroup \smaller\smaller\smaller\begin{tabular}{@{}c@{}}%
1\\0\\-1
\end{tabular}\endgroup%
{$\left.\llap{\phantom{%
\begingroup \smaller\smaller\smaller\begin{tabular}{@{}c@{}}%
0\\0\\0
\end{tabular}\endgroup%
}}\!\right]$}%
}%
\ifdim\wd\matricesbox>\halfwidth\myboxwidth=\hsize\else\myboxwidth=\halfwidth\fi
\vbox{%
\ifdim\myboxwidth=\hsize
\setbox\onelinebox=\hbox{%
\vbox{\hbox{%
$\Pi_{4,55}=\hbox{GN}_{28}$ spans $L_{1.6}$%
}\hbox{%
$\infty\infty2\infty$ (shared)%
}%
}%
\hfill\copy\matricesbox
}%
\ifdim\wd\onelinebox>\myboxwidth
\hbox to \myboxwidth{%
$\Pi_{4,55}=\hbox{GN}_{28}$ spans $L_{1.6}$%
\hfil
$\infty\infty2\infty$ (shared)%
}%
\box\matricesbox
\else
\hbox to \myboxwidth{%
\unhbox\onelinebox
}%
\fi
\else
\hbox to \myboxwidth{%
$\Pi_{4,55}=\hbox{GN}_{28}$ spans $L_{1.6}$%
\hfil}%
\hbox to \myboxwidth{%
$\infty\infty2\infty$ (shared)%
\hfil}%
\box\matricesbox
\fi
}%
\hfill\discretionary{}{}{}%
\setbox\matricesbox=\hbox{%
{$\left[\!\llap{\phantom{%
\begingroup \smaller\smaller\smaller\begin{tabular}{@{}c@{}}%
\phantom{0}\\\phantom{0}\\\phantom{0}
\end{tabular}\endgroup%
}}\right.$}%
\begingroup \smaller\smaller\smaller\begin{tabular}{@{}c@{}}%
-1/9\\\phantom{0}\\\phantom{0}
\end{tabular}\endgroup%
\kern3pt%
\begingroup \smaller\smaller\smaller\begin{tabular}{@{}c@{}}%
\phantom{0}\\4/9\\-2/9
\end{tabular}\endgroup%
\kern3pt%
\begingroup \smaller\smaller\smaller\begin{tabular}{@{}c@{}}%
\phantom{0}\\-2/9\\10/9
\end{tabular}\endgroup%
{$\left.\llap{\phantom{%
\begingroup \smaller\smaller\smaller\begin{tabular}{@{}c@{}}%
\phantom{0}\\\phantom{0}\\\phantom{0}
\end{tabular}\endgroup%
}}\!\right]$}%
{$\left[\!\llap{\phantom{%
\begingroup \smaller\smaller\smaller\begin{tabular}{@{}c@{}}%
0\\0\\0
\end{tabular}\endgroup%
}}\right.$}%
\begingroup \smaller\smaller\smaller\begin{tabular}{@{}c@{}}%
8\\6\\2
\end{tabular}\endgroup%
\kern3pt%
\begingroup \smaller\smaller\smaller\begin{tabular}{@{}c@{}}%
4\\-1\\2
\end{tabular}\endgroup%
\kern3pt%
\begingroup \smaller\smaller\smaller\begin{tabular}{@{}c@{}}%
1\\-1\\-1
\end{tabular}\endgroup%
\kern3pt%
\begingroup \smaller\smaller\smaller\begin{tabular}{@{}c@{}}%
8\\4\\-2
\end{tabular}\endgroup%
{$\left.\llap{\phantom{%
\begingroup \smaller\smaller\smaller\begin{tabular}{@{}c@{}}%
0\\0\\0
\end{tabular}\endgroup%
}}\!\right]$}%
}%
\ifdim\wd\matricesbox>\halfwidth\myboxwidth=\hsize\else\myboxwidth=\halfwidth\fi
\vbox{%
\ifdim\myboxwidth=\hsize
\setbox\onelinebox=\hbox{%
\vbox{\hbox{%
$\Pi_{4,56}$ spans $L_{1.9}$%
}\hbox{%
$4\infty22$ (shared)%
}%
}%
\hfill\copy\matricesbox
}%
\ifdim\wd\onelinebox>\myboxwidth
\hbox to \myboxwidth{%
$\Pi_{4,56}$ spans $L_{1.9}$%
\hfil
$4\infty22$ (shared)%
}%
\box\matricesbox
\else
\hbox to \myboxwidth{%
\unhbox\onelinebox
}%
\fi
\else
\hbox to \myboxwidth{%
$\Pi_{4,56}$ spans $L_{1.9}$%
\hfil}%
\hbox to \myboxwidth{%
$4\infty22$ (shared)%
\hfil}%
\box\matricesbox
\fi
}%
\hfill\discretionary{}{}{}%
\setbox\matricesbox=\hbox{%
{$\left[\!\llap{\phantom{%
\begingroup \smaller\smaller\smaller\begin{tabular}{@{}c@{}}%
\phantom{0}\\\phantom{0}\\\phantom{0}
\end{tabular}\endgroup%
}}\right.$}%
\begingroup \smaller\smaller\smaller\begin{tabular}{@{}c@{}}%
-1/12\\\phantom{0}\\\phantom{0}
\end{tabular}\endgroup%
\kern3pt%
\begingroup \smaller\smaller\smaller\begin{tabular}{@{}c@{}}%
\phantom{0}\\13/12\\-5/12
\end{tabular}\endgroup%
\kern3pt%
\begingroup \smaller\smaller\smaller\begin{tabular}{@{}c@{}}%
\phantom{0}\\-5/12\\13/12
\end{tabular}\endgroup%
{$\left.\llap{\phantom{%
\begingroup \smaller\smaller\smaller\begin{tabular}{@{}c@{}}%
\phantom{0}\\\phantom{0}\\\phantom{0}
\end{tabular}\endgroup%
}}\!\right]$}%
{$\left[\!\llap{\phantom{%
\begingroup \smaller\smaller\smaller\begin{tabular}{@{}c@{}}%
0\\0\\0
\end{tabular}\endgroup%
}}\right.$}%
\begingroup \smaller\smaller\smaller\begin{tabular}{@{}c@{}}%
3\\-2\\-1
\end{tabular}\endgroup%
\kern3pt%
\begingroup \smaller\smaller\smaller\begin{tabular}{@{}c@{}}%
12\\1\\5
\end{tabular}\endgroup%
\kern3pt%
\begingroup \smaller\smaller\smaller\begin{tabular}{@{}c@{}}%
12\\5\\1
\end{tabular}\endgroup%
\kern3pt%
\begingroup \smaller\smaller\smaller\begin{tabular}{@{}c@{}}%
1\\0\\-1
\end{tabular}\endgroup%
{$\left.\llap{\phantom{%
\begingroup \smaller\smaller\smaller\begin{tabular}{@{}c@{}}%
0\\0\\0
\end{tabular}\endgroup%
}}\!\right]$}%
}%
\ifdim\wd\matricesbox>\halfwidth\myboxwidth=\hsize\else\myboxwidth=\halfwidth\fi
\vbox{%
\ifdim\myboxwidth=\hsize
\setbox\onelinebox=\hbox{%
\vbox{\hbox{%
$\Pi_{4,57}$ spans $L_{144.8}$%
}\hbox{%
$\infty\infty22$ (shared)%
}%
}%
\hfill\copy\matricesbox
}%
\ifdim\wd\onelinebox>\myboxwidth
\hbox to \myboxwidth{%
$\Pi_{4,57}$ spans $L_{144.8}$%
\hfil
$\infty\infty22$ (shared)%
}%
\box\matricesbox
\else
\hbox to \myboxwidth{%
\unhbox\onelinebox
}%
\fi
\else
\hbox to \myboxwidth{%
$\Pi_{4,57}$ spans $L_{144.8}$%
\hfil}%
\hbox to \myboxwidth{%
$\infty\infty22$ (shared)%
\hfil}%
\box\matricesbox
\fi
}%
\hfill\discretionary{}{}{}%
\setbox\matricesbox=\hbox{%
{$\left[\!\llap{\phantom{%
\begingroup \smaller\smaller\smaller\begin{tabular}{@{}c@{}}%
\phantom{0}\\\phantom{0}\\\phantom{0}
\end{tabular}\endgroup%
}}\right.$}%
\begingroup \smaller\smaller\smaller\begin{tabular}{@{}c@{}}%
-1/4\\\phantom{0}\\\phantom{0}
\end{tabular}\endgroup%
\kern3pt%
\begingroup \smaller\smaller\smaller\begin{tabular}{@{}c@{}}%
\phantom{0}\\1\\-1/2
\end{tabular}\endgroup%
\kern3pt%
\begingroup \smaller\smaller\smaller\begin{tabular}{@{}c@{}}%
\phantom{0}\\-1/2\\5/4
\end{tabular}\endgroup%
{$\left.\llap{\phantom{%
\begingroup \smaller\smaller\smaller\begin{tabular}{@{}c@{}}%
\phantom{0}\\\phantom{0}\\\phantom{0}
\end{tabular}\endgroup%
}}\!\right]$}%
{$\left[\!\llap{\phantom{%
\begingroup \smaller\smaller\smaller\begin{tabular}{@{}c@{}}%
0\\0\\0
\end{tabular}\endgroup%
}}\right.$}%
\begingroup \smaller\smaller\smaller\begin{tabular}{@{}c@{}}%
4\\3\\2
\end{tabular}\endgroup%
\kern3pt%
\begingroup \smaller\smaller\smaller\begin{tabular}{@{}c@{}}%
4\\-1\\2
\end{tabular}\endgroup%
\kern3pt%
\begingroup \smaller\smaller\smaller\begin{tabular}{@{}c@{}}%
1\\-1\\-1
\end{tabular}\endgroup%
\kern3pt%
\begingroup \smaller\smaller\smaller\begin{tabular}{@{}c@{}}%
4\\1\\-2
\end{tabular}\endgroup%
{$\left.\llap{\phantom{%
\begingroup \smaller\smaller\smaller\begin{tabular}{@{}c@{}}%
0\\0\\0
\end{tabular}\endgroup%
}}\!\right]$}%
}%
\ifdim\wd\matricesbox>\halfwidth\myboxwidth=\hsize\else\myboxwidth=\halfwidth\fi
\vbox{%
\ifdim\myboxwidth=\hsize
\setbox\onelinebox=\hbox{%
\vbox{\hbox{%
$\Pi_{4,58}=\hbox{GN}_{24}$ spans $L_{140.3}$%
}\hbox{%
$\infty\infty2\infty$ (shared)%
}%
}%
\hfill\copy\matricesbox
}%
\ifdim\wd\onelinebox>\myboxwidth
\hbox to \myboxwidth{%
$\Pi_{4,58}=\hbox{GN}_{24}$ spans $L_{140.3}$%
\hfil
$\infty\infty2\infty$ (shared)%
}%
\box\matricesbox
\else
\hbox to \myboxwidth{%
\unhbox\onelinebox
}%
\fi
\else
\hbox to \myboxwidth{%
$\Pi_{4,58}=\hbox{GN}_{24}$ spans $L_{140.3}$%
\hfil}%
\hbox to \myboxwidth{%
$\infty\infty2\infty$ (shared)%
\hfil}%
\box\matricesbox
\fi
}%
\hfill\discretionary{}{}{}%
\setbox\matricesbox=\hbox{%
{$\left[\!\llap{\phantom{%
\begingroup \smaller\smaller\smaller\begin{tabular}{@{}c@{}}%
\phantom{0}\\\phantom{0}\\\phantom{0}
\end{tabular}\endgroup%
}}\right.$}%
\begingroup \smaller\smaller\smaller\begin{tabular}{@{}c@{}}%
-1/8\\\phantom{0}\\\phantom{0}
\end{tabular}\endgroup%
\kern3pt%
\begingroup \smaller\smaller\smaller\begin{tabular}{@{}c@{}}%
\phantom{0}\\1/2\\-1/4
\end{tabular}\endgroup%
\kern3pt%
\begingroup \smaller\smaller\smaller\begin{tabular}{@{}c@{}}%
\phantom{0}\\-1/4\\9/8
\end{tabular}\endgroup%
{$\left.\llap{\phantom{%
\begingroup \smaller\smaller\smaller\begin{tabular}{@{}c@{}}%
\phantom{0}\\\phantom{0}\\\phantom{0}
\end{tabular}\endgroup%
}}\!\right]$}%
{$\left[\!\llap{\phantom{%
\begingroup \smaller\smaller\smaller\begin{tabular}{@{}c@{}}%
0\\0\\0
\end{tabular}\endgroup%
}}\right.$}%
\begingroup \smaller\smaller\smaller\begin{tabular}{@{}c@{}}%
16\\6\\-4
\end{tabular}\endgroup%
\kern3pt%
\begingroup \smaller\smaller\smaller\begin{tabular}{@{}c@{}}%
4\\-3\\-2
\end{tabular}\endgroup%
\kern3pt%
\begingroup \smaller\smaller\smaller\begin{tabular}{@{}c@{}}%
1\\0\\1
\end{tabular}\endgroup%
\kern3pt%
\begingroup \smaller\smaller\smaller\begin{tabular}{@{}c@{}}%
16\\10\\4
\end{tabular}\endgroup%
{$\left.\llap{\phantom{%
\begingroup \smaller\smaller\smaller\begin{tabular}{@{}c@{}}%
0\\0\\0
\end{tabular}\endgroup%
}}\!\right]$}%
}%
\ifdim\wd\matricesbox>\halfwidth\myboxwidth=\hsize\else\myboxwidth=\halfwidth\fi
\vbox{%
\ifdim\myboxwidth=\hsize
\setbox\onelinebox=\hbox{%
\vbox{\hbox{%
$\Pi_{4,59}=\hbox{GN}_{16}$ spans $L_{140.4}$%
}\hbox{%
$\infty\infty2\infty$ (shared)%
}%
}%
\hfill\copy\matricesbox
}%
\ifdim\wd\onelinebox>\myboxwidth
\hbox to \myboxwidth{%
$\Pi_{4,59}=\hbox{GN}_{16}$ spans $L_{140.4}$%
\hfil
$\infty\infty2\infty$ (shared)%
}%
\box\matricesbox
\else
\hbox to \myboxwidth{%
\unhbox\onelinebox
}%
\fi
\else
\hbox to \myboxwidth{%
$\Pi_{4,59}=\hbox{GN}_{16}$ spans $L_{140.4}$%
\hfil}%
\hbox to \myboxwidth{%
$\infty\infty2\infty$ (shared)%
\hfil}%
\box\matricesbox
\fi
}%
\hfill\discretionary{}{}{}%
\setbox\matricesbox=\hbox{%
{$\left[\!\llap{\phantom{%
\begingroup \smaller\smaller\smaller\begin{tabular}{@{}c@{}}%
\phantom{0}\\\phantom{0}\\\phantom{0}
\end{tabular}\endgroup%
}}\right.$}%
\begingroup \smaller\smaller\smaller\begin{tabular}{@{}c@{}}%
-1/8\\\phantom{0}\\\phantom{0}
\end{tabular}\endgroup%
\kern3pt%
\begingroup \smaller\smaller\smaller\begin{tabular}{@{}c@{}}%
\phantom{0}\\9/8\\-3/8
\end{tabular}\endgroup%
\kern3pt%
\begingroup \smaller\smaller\smaller\begin{tabular}{@{}c@{}}%
\phantom{0}\\-3/8\\33/8
\end{tabular}\endgroup%
{$\left.\llap{\phantom{%
\begingroup \smaller\smaller\smaller\begin{tabular}{@{}c@{}}%
\phantom{0}\\\phantom{0}\\\phantom{0}
\end{tabular}\endgroup%
}}\!\right]$}%
{$\left[\!\llap{\phantom{%
\begingroup \smaller\smaller\smaller\begin{tabular}{@{}c@{}}%
0\\0\\0
\end{tabular}\endgroup%
}}\right.$}%
\begingroup \smaller\smaller\smaller\begin{tabular}{@{}c@{}}%
1\\1\\0
\end{tabular}\endgroup%
\kern3pt%
\begingroup \smaller\smaller\smaller\begin{tabular}{@{}c@{}}%
4\\-1\\1
\end{tabular}\endgroup%
\kern3pt%
\begingroup \smaller\smaller\smaller\begin{tabular}{@{}c@{}}%
16\\-6\\-2
\end{tabular}\endgroup%
\kern3pt%
\begingroup \smaller\smaller\smaller\begin{tabular}{@{}c@{}}%
3\\0\\-1
\end{tabular}\endgroup%
{$\left.\llap{\phantom{%
\begingroup \smaller\smaller\smaller\begin{tabular}{@{}c@{}}%
0\\0\\0
\end{tabular}\endgroup%
}}\!\right]$}%
}%
\ifdim\wd\matricesbox>\halfwidth\myboxwidth=\hsize\else\myboxwidth=\halfwidth\fi
\vbox{%
\ifdim\myboxwidth=\hsize
\setbox\onelinebox=\hbox{%
\vbox{\hbox{%
$\Pi_{4,60}$ spans $L_{144.5}$%
}\hbox{%
$\infty\infty22$ (shared)%
}%
}%
\hfill\copy\matricesbox
}%
\ifdim\wd\onelinebox>\myboxwidth
\hbox to \myboxwidth{%
$\Pi_{4,60}$ spans $L_{144.5}$%
\hfil
$\infty\infty22$ (shared)%
}%
\box\matricesbox
\else
\hbox to \myboxwidth{%
\unhbox\onelinebox
}%
\fi
\else
\hbox to \myboxwidth{%
$\Pi_{4,60}$ spans $L_{144.5}$%
\hfil}%
\hbox to \myboxwidth{%
$\infty\infty22$ (shared)%
\hfil}%
\box\matricesbox
\fi
}%
\hfill\discretionary{}{}{}%
\setbox\matricesbox=\hbox{%
{$\left[\!\llap{\phantom{%
\begingroup \smaller\smaller\smaller\begin{tabular}{@{}c@{}}%
\phantom{0}\\\phantom{0}\\\phantom{0}
\end{tabular}\endgroup%
}}\right.$}%
\begingroup \smaller\smaller\smaller\begin{tabular}{@{}c@{}}%
-1/9\\\phantom{0}\\\phantom{0}
\end{tabular}\endgroup%
\kern3pt%
\begingroup \smaller\smaller\smaller\begin{tabular}{@{}c@{}}%
\phantom{0}\\4/9\\-2/9
\end{tabular}\endgroup%
\kern3pt%
\begingroup \smaller\smaller\smaller\begin{tabular}{@{}c@{}}%
\phantom{0}\\-2/9\\28/9
\end{tabular}\endgroup%
{$\left.\llap{\phantom{%
\begingroup \smaller\smaller\smaller\begin{tabular}{@{}c@{}}%
\phantom{0}\\\phantom{0}\\\phantom{0}
\end{tabular}\endgroup%
}}\!\right]$}%
{$\left[\!\llap{\phantom{%
\begingroup \smaller\smaller\smaller\begin{tabular}{@{}c@{}}%
0\\0\\0
\end{tabular}\endgroup%
}}\right.$}%
\begingroup \smaller\smaller\smaller\begin{tabular}{@{}c@{}}%
3\\1\\-1
\end{tabular}\endgroup%
\kern3pt%
\begingroup \smaller\smaller\smaller\begin{tabular}{@{}c@{}}%
3\\2\\1
\end{tabular}\endgroup%
\kern3pt%
\begingroup \smaller\smaller\smaller\begin{tabular}{@{}c@{}}%
12\\-7\\1
\end{tabular}\endgroup%
\kern3pt%
\begingroup \smaller\smaller\smaller\begin{tabular}{@{}c@{}}%
4\\-3\\-1
\end{tabular}\endgroup%
{$\left.\llap{\phantom{%
\begingroup \smaller\smaller\smaller\begin{tabular}{@{}c@{}}%
0\\0\\0
\end{tabular}\endgroup%
}}\!\right]$}%
}%
\ifdim\wd\matricesbox>\halfwidth\myboxwidth=\hsize\else\myboxwidth=\halfwidth\fi
\vbox{%
\ifdim\myboxwidth=\hsize
\setbox\onelinebox=\hbox{%
\vbox{\hbox{%
$\Pi_{4,61}$ spans $L_{7.11}$%
}\hbox{%
$\infty\infty22$ (shared)%
}%
}%
\hfill\copy\matricesbox
}%
\ifdim\wd\onelinebox>\myboxwidth
\hbox to \myboxwidth{%
$\Pi_{4,61}$ spans $L_{7.11}$%
\hfil
$\infty\infty22$ (shared)%
}%
\box\matricesbox
\else
\hbox to \myboxwidth{%
\unhbox\onelinebox
}%
\fi
\else
\hbox to \myboxwidth{%
$\Pi_{4,61}$ spans $L_{7.11}$%
\hfil}%
\hbox to \myboxwidth{%
$\infty\infty22$ (shared)%
\hfil}%
\box\matricesbox
\fi
}%
\hfill\discretionary{}{}{}%
\setbox\matricesbox=\hbox{%
{$\left[\!\llap{\phantom{%
\begingroup \smaller\smaller\smaller\begin{tabular}{@{}c@{}}%
\phantom{0}\\\phantom{0}\\\phantom{0}
\end{tabular}\endgroup%
}}\right.$}%
\begingroup \smaller\smaller\smaller\begin{tabular}{@{}c@{}}%
-1/16\\\phantom{0}\\\phantom{0}
\end{tabular}\endgroup%
\kern3pt%
\begingroup \smaller\smaller\smaller\begin{tabular}{@{}c@{}}%
\phantom{0}\\9/16\\-1/4
\end{tabular}\endgroup%
\kern3pt%
\begingroup \smaller\smaller\smaller\begin{tabular}{@{}c@{}}%
\phantom{0}\\-1/4\\1
\end{tabular}\endgroup%
{$\left.\llap{\phantom{%
\begingroup \smaller\smaller\smaller\begin{tabular}{@{}c@{}}%
\phantom{0}\\\phantom{0}\\\phantom{0}
\end{tabular}\endgroup%
}}\!\right]$}%
{$\left[\!\llap{\phantom{%
\begingroup \smaller\smaller\smaller\begin{tabular}{@{}c@{}}%
0\\0\\0
\end{tabular}\endgroup%
}}\right.$}%
\begingroup \smaller\smaller\smaller\begin{tabular}{@{}c@{}}%
8\\-4\\1
\end{tabular}\endgroup%
\kern3pt%
\begingroup \smaller\smaller\smaller\begin{tabular}{@{}c@{}}%
8\\4\\3
\end{tabular}\endgroup%
\kern3pt%
\begingroup \smaller\smaller\smaller\begin{tabular}{@{}c@{}}%
32\\8\\-6
\end{tabular}\endgroup%
\kern3pt%
\begingroup \smaller\smaller\smaller\begin{tabular}{@{}c@{}}%
1\\-1\\-1
\end{tabular}\endgroup%
{$\left.\llap{\phantom{%
\begingroup \smaller\smaller\smaller\begin{tabular}{@{}c@{}}%
0\\0\\0
\end{tabular}\endgroup%
}}\!\right]$}%
}%
\ifdim\wd\matricesbox>\halfwidth\myboxwidth=\hsize\else\myboxwidth=\halfwidth\fi
\vbox{%
\ifdim\myboxwidth=\hsize
\setbox\onelinebox=\hbox{%
\vbox{\hbox{%
$\Pi_{4,62}$ spans $L_{141.10}$%
}\hbox{%
$\infty\infty22$ (shared)%
}%
}%
\hfill\copy\matricesbox
}%
\ifdim\wd\onelinebox>\myboxwidth
\hbox to \myboxwidth{%
$\Pi_{4,62}$ spans $L_{141.10}$%
\hfil
$\infty\infty22$ (shared)%
}%
\box\matricesbox
\else
\hbox to \myboxwidth{%
\unhbox\onelinebox
}%
\fi
\else
\hbox to \myboxwidth{%
$\Pi_{4,62}$ spans $L_{141.10}$%
\hfil}%
\hbox to \myboxwidth{%
$\infty\infty22$ (shared)%
\hfil}%
\box\matricesbox
\fi
}%
\hfill\discretionary{}{}{}%
\setbox\matricesbox=\hbox{%
{$\left[\!\llap{\phantom{%
\begingroup \smaller\smaller\smaller\begin{tabular}{@{}c@{}}%
\phantom{0}\\\phantom{0}\\\phantom{0}
\end{tabular}\endgroup%
}}\right.$}%
\begingroup \smaller\smaller\smaller\begin{tabular}{@{}c@{}}%
-1/68\\\phantom{0}\\\phantom{0}
\end{tabular}\endgroup%
\kern3pt%
\begingroup \smaller\smaller\smaller\begin{tabular}{@{}c@{}}%
\phantom{0}\\15/17\\-5/17
\end{tabular}\endgroup%
\kern3pt%
\begingroup \smaller\smaller\smaller\begin{tabular}{@{}c@{}}%
\phantom{0}\\-5/17\\30/17
\end{tabular}\endgroup%
{$\left.\llap{\phantom{%
\begingroup \smaller\smaller\smaller\begin{tabular}{@{}c@{}}%
\phantom{0}\\\phantom{0}\\\phantom{0}
\end{tabular}\endgroup%
}}\!\right]$}%
{$\left[\!\llap{\phantom{%
\begingroup \smaller\smaller\smaller\begin{tabular}{@{}c@{}}%
0\\0\\0
\end{tabular}\endgroup%
}}\right.$}%
\begingroup \smaller\smaller\smaller\begin{tabular}{@{}c@{}}%
10\\3\\-1
\end{tabular}\endgroup%
\kern3pt%
\begingroup \smaller\smaller\smaller\begin{tabular}{@{}c@{}}%
10\\-3\\-2
\end{tabular}\endgroup%
\kern3pt%
\begingroup \smaller\smaller\smaller\begin{tabular}{@{}c@{}}%
20\\-4\\2
\end{tabular}\endgroup%
\kern3pt%
\begingroup \smaller\smaller\smaller\begin{tabular}{@{}c@{}}%
2\\1\\1
\end{tabular}\endgroup%
{$\left.\llap{\phantom{%
\begingroup \smaller\smaller\smaller\begin{tabular}{@{}c@{}}%
0\\0\\0
\end{tabular}\endgroup%
}}\!\right]$}%
}%
\ifdim\wd\matricesbox>\halfwidth\myboxwidth=\hsize\else\myboxwidth=\halfwidth\fi
\vbox{%
\ifdim\myboxwidth=\hsize
\setbox\onelinebox=\hbox{%
\vbox{\hbox{%
$\Pi_{4,63}$ spans $L_{6.5}$%
}\hbox{%
$3222$%
}%
}%
\hfill\copy\matricesbox
}%
\ifdim\wd\onelinebox>\myboxwidth
\hbox to \myboxwidth{%
$\Pi_{4,63}$ spans $L_{6.5}$%
\hfil
$3222$%
}%
\box\matricesbox
\else
\hbox to \myboxwidth{%
\unhbox\onelinebox
}%
\fi
\else
\hbox to \myboxwidth{%
$\Pi_{4,63}$ spans $L_{6.5}$%
\hfil}%
\hbox to \myboxwidth{%
$3222$%
\hfil}%
\box\matricesbox
\fi
}%
\hfill\discretionary{}{}{}%
\setbox\matricesbox=\hbox{%
{$\left[\!\llap{\phantom{%
\begingroup \smaller\smaller\smaller\begin{tabular}{@{}c@{}}%
\phantom{0}\\\phantom{0}\\\phantom{0}
\end{tabular}\endgroup%
}}\right.$}%
\begingroup \smaller\smaller\smaller\begin{tabular}{@{}c@{}}%
-1/40\\\phantom{0}\\\phantom{0}
\end{tabular}\endgroup%
\kern3pt%
\begingroup \smaller\smaller\smaller\begin{tabular}{@{}c@{}}%
\phantom{0}\\21/10\\-3/10
\end{tabular}\endgroup%
\kern3pt%
\begingroup \smaller\smaller\smaller\begin{tabular}{@{}c@{}}%
\phantom{0}\\-3/10\\29/10
\end{tabular}\endgroup%
{$\left.\llap{\phantom{%
\begingroup \smaller\smaller\smaller\begin{tabular}{@{}c@{}}%
\phantom{0}\\\phantom{0}\\\phantom{0}
\end{tabular}\endgroup%
}}\!\right]$}%
{$\left[\!\llap{\phantom{%
\begingroup \smaller\smaller\smaller\begin{tabular}{@{}c@{}}%
0\\0\\0
\end{tabular}\endgroup%
}}\right.$}%
\begingroup \smaller\smaller\smaller\begin{tabular}{@{}c@{}}%
2\\1\\0
\end{tabular}\endgroup%
\kern3pt%
\begingroup \smaller\smaller\smaller\begin{tabular}{@{}c@{}}%
4\\-1\\-1
\end{tabular}\endgroup%
\kern3pt%
\begingroup \smaller\smaller\smaller\begin{tabular}{@{}c@{}}%
10\\-2\\1
\end{tabular}\endgroup%
\kern3pt%
\begingroup \smaller\smaller\smaller\begin{tabular}{@{}c@{}}%
48\\2\\6
\end{tabular}\endgroup%
{$\left.\llap{\phantom{%
\begingroup \smaller\smaller\smaller\begin{tabular}{@{}c@{}}%
0\\0\\0
\end{tabular}\endgroup%
}}\!\right]$}%
}%
\ifdim\wd\matricesbox>\halfwidth\myboxwidth=\hsize\else\myboxwidth=\halfwidth\fi
\vbox{%
\ifdim\myboxwidth=\hsize
\setbox\onelinebox=\hbox{%
\vbox{\hbox{%
$\Pi_{4,64}$ spans $L_{19.1}$%
}\hbox{%
$4222$ (shared)%
}%
}%
\hfill\copy\matricesbox
}%
\ifdim\wd\onelinebox>\myboxwidth
\hbox to \myboxwidth{%
$\Pi_{4,64}$ spans $L_{19.1}$%
\hfil
$4222$ (shared)%
}%
\box\matricesbox
\else
\hbox to \myboxwidth{%
\unhbox\onelinebox
}%
\fi
\else
\hbox to \myboxwidth{%
$\Pi_{4,64}$ spans $L_{19.1}$%
\hfil}%
\hbox to \myboxwidth{%
$4222$ (shared)%
\hfil}%
\box\matricesbox
\fi
}%
\hfill\discretionary{}{}{}%
\setbox\matricesbox=\hbox{%
{$\left[\!\llap{\phantom{%
\begingroup \smaller\smaller\smaller\begin{tabular}{@{}c@{}}%
\phantom{0}\\\phantom{0}\\\phantom{0}
\end{tabular}\endgroup%
}}\right.$}%
\begingroup \smaller\smaller\smaller\begin{tabular}{@{}c@{}}%
-1/14\\\phantom{0}\\\phantom{0}
\end{tabular}\endgroup%
\kern3pt%
\begingroup \smaller\smaller\smaller\begin{tabular}{@{}c@{}}%
\phantom{0}\\4/7\\-1/7
\end{tabular}\endgroup%
\kern3pt%
\begingroup \smaller\smaller\smaller\begin{tabular}{@{}c@{}}%
\phantom{0}\\-1/7\\11/14
\end{tabular}\endgroup%
{$\left.\llap{\phantom{%
\begingroup \smaller\smaller\smaller\begin{tabular}{@{}c@{}}%
\phantom{0}\\\phantom{0}\\\phantom{0}
\end{tabular}\endgroup%
}}\!\right]$}%
{$\left[\!\llap{\phantom{%
\begingroup \smaller\smaller\smaller\begin{tabular}{@{}c@{}}%
0\\0\\0
\end{tabular}\endgroup%
}}\right.$}%
\begingroup \smaller\smaller\smaller\begin{tabular}{@{}c@{}}%
1\\1\\1
\end{tabular}\endgroup%
\kern3pt%
\begingroup \smaller\smaller\smaller\begin{tabular}{@{}c@{}}%
2\\-2\\0
\end{tabular}\endgroup%
\kern3pt%
\begingroup \smaller\smaller\smaller\begin{tabular}{@{}c@{}}%
8\\-2\\-4
\end{tabular}\endgroup%
\kern3pt%
\begingroup \smaller\smaller\smaller\begin{tabular}{@{}c@{}}%
3\\2\\-1
\end{tabular}\endgroup%
{$\left.\llap{\phantom{%
\begingroup \smaller\smaller\smaller\begin{tabular}{@{}c@{}}%
0\\0\\0
\end{tabular}\endgroup%
}}\!\right]$}%
}%
\ifdim\wd\matricesbox>\halfwidth\myboxwidth=\hsize\else\myboxwidth=\halfwidth\fi
\vbox{%
\ifdim\myboxwidth=\hsize
\setbox\onelinebox=\hbox{%
\vbox{\hbox{%
$\Pi_{4,65}$ spans $L_{123.3}$%
}\hbox{%
$4222$ (shared)%
}%
}%
\hfill\copy\matricesbox
}%
\ifdim\wd\onelinebox>\myboxwidth
\hbox to \myboxwidth{%
$\Pi_{4,65}$ spans $L_{123.3}$%
\hfil
$4222$ (shared)%
}%
\box\matricesbox
\else
\hbox to \myboxwidth{%
\unhbox\onelinebox
}%
\fi
\else
\hbox to \myboxwidth{%
$\Pi_{4,65}$ spans $L_{123.3}$%
\hfil}%
\hbox to \myboxwidth{%
$4222$ (shared)%
\hfil}%
\box\matricesbox
\fi
}%
\hfill\discretionary{}{}{}%
\setbox\matricesbox=\hbox{%
{$\left[\!\llap{\phantom{%
\begingroup \smaller\smaller\smaller\begin{tabular}{@{}c@{}}%
\phantom{0}\\\phantom{0}\\\phantom{0}
\end{tabular}\endgroup%
}}\right.$}%
\begingroup \smaller\smaller\smaller\begin{tabular}{@{}c@{}}%
-1/10\\\phantom{0}\\\phantom{0}
\end{tabular}\endgroup%
\kern3pt%
\begingroup \smaller\smaller\smaller\begin{tabular}{@{}c@{}}%
\phantom{0}\\11/10\\-3/10
\end{tabular}\endgroup%
\kern3pt%
\begingroup \smaller\smaller\smaller\begin{tabular}{@{}c@{}}%
\phantom{0}\\-3/10\\19/10
\end{tabular}\endgroup%
{$\left.\llap{\phantom{%
\begingroup \smaller\smaller\smaller\begin{tabular}{@{}c@{}}%
\phantom{0}\\\phantom{0}\\\phantom{0}
\end{tabular}\endgroup%
}}\!\right]$}%
{$\left[\!\llap{\phantom{%
\begingroup \smaller\smaller\smaller\begin{tabular}{@{}c@{}}%
0\\0\\0
\end{tabular}\endgroup%
}}\right.$}%
\begingroup \smaller\smaller\smaller\begin{tabular}{@{}c@{}}%
1\\-1\\0
\end{tabular}\endgroup%
\kern3pt%
\begingroup \smaller\smaller\smaller\begin{tabular}{@{}c@{}}%
2\\1\\1
\end{tabular}\endgroup%
\kern3pt%
\begingroup \smaller\smaller\smaller\begin{tabular}{@{}c@{}}%
20\\7\\-1
\end{tabular}\endgroup%
\kern3pt%
\begingroup \smaller\smaller\smaller\begin{tabular}{@{}c@{}}%
5\\-1\\-2
\end{tabular}\endgroup%
{$\left.\llap{\phantom{%
\begingroup \smaller\smaller\smaller\begin{tabular}{@{}c@{}}%
0\\0\\0
\end{tabular}\endgroup%
}}\!\right]$}%
}%
\ifdim\wd\matricesbox>\halfwidth\myboxwidth=\hsize\else\myboxwidth=\halfwidth\fi
\vbox{%
\ifdim\myboxwidth=\hsize
\setbox\onelinebox=\hbox{%
\vbox{\hbox{%
$\Pi_{4,66}$ spans $L_{5.6}$%
}\hbox{%
$42\infty2$%
}%
}%
\hfill\copy\matricesbox
}%
\ifdim\wd\onelinebox>\myboxwidth
\hbox to \myboxwidth{%
$\Pi_{4,66}$ spans $L_{5.6}$%
\hfil
$42\infty2$%
}%
\box\matricesbox
\else
\hbox to \myboxwidth{%
\unhbox\onelinebox
}%
\fi
\else
\hbox to \myboxwidth{%
$\Pi_{4,66}$ spans $L_{5.6}$%
\hfil}%
\hbox to \myboxwidth{%
$42\infty2$%
\hfil}%
\box\matricesbox
\fi
}%
\hfill\discretionary{}{}{}%
\setbox\matricesbox=\hbox{%
{$\left[\!\llap{\phantom{%
\begingroup \smaller\smaller\smaller\begin{tabular}{@{}c@{}}%
\phantom{0}\\\phantom{0}\\\phantom{0}
\end{tabular}\endgroup%
}}\right.$}%
\begingroup \smaller\smaller\smaller\begin{tabular}{@{}c@{}}%
-1/20\\\phantom{0}\\\phantom{0}
\end{tabular}\endgroup%
\kern3pt%
\begingroup \smaller\smaller\smaller\begin{tabular}{@{}c@{}}%
\phantom{0}\\21/20\\-9/20
\end{tabular}\endgroup%
\kern3pt%
\begingroup \smaller\smaller\smaller\begin{tabular}{@{}c@{}}%
\phantom{0}\\-9/20\\21/20
\end{tabular}\endgroup%
{$\left.\llap{\phantom{%
\begingroup \smaller\smaller\smaller\begin{tabular}{@{}c@{}}%
\phantom{0}\\\phantom{0}\\\phantom{0}
\end{tabular}\endgroup%
}}\!\right]$}%
{$\left[\!\llap{\phantom{%
\begingroup \smaller\smaller\smaller\begin{tabular}{@{}c@{}}%
0\\0\\0
\end{tabular}\endgroup%
}}\right.$}%
\begingroup \smaller\smaller\smaller\begin{tabular}{@{}c@{}}%
3\\1\\2
\end{tabular}\endgroup%
\kern3pt%
\begingroup \smaller\smaller\smaller\begin{tabular}{@{}c@{}}%
6\\-3\\-1
\end{tabular}\endgroup%
\kern3pt%
\begingroup \smaller\smaller\smaller\begin{tabular}{@{}c@{}}%
6\\-1\\-3
\end{tabular}\endgroup%
\kern3pt%
\begingroup \smaller\smaller\smaller\begin{tabular}{@{}c@{}}%
1\\1\\0
\end{tabular}\endgroup%
{$\left.\llap{\phantom{%
\begingroup \smaller\smaller\smaller\begin{tabular}{@{}c@{}}%
0\\0\\0
\end{tabular}\endgroup%
}}\!\right]$}%
}%
\ifdim\wd\matricesbox>\halfwidth\myboxwidth=\hsize\else\myboxwidth=\halfwidth\fi
\vbox{%
\ifdim\myboxwidth=\hsize
\setbox\onelinebox=\hbox{%
\vbox{\hbox{%
$\Pi_{4,67}$ spans $L_{123.6}$%
}\hbox{%
$4222$ (shared)%
}%
}%
\hfill\copy\matricesbox
}%
\ifdim\wd\onelinebox>\myboxwidth
\hbox to \myboxwidth{%
$\Pi_{4,67}$ spans $L_{123.6}$%
\hfil
$4222$ (shared)%
}%
\box\matricesbox
\else
\hbox to \myboxwidth{%
\unhbox\onelinebox
}%
\fi
\else
\hbox to \myboxwidth{%
$\Pi_{4,67}$ spans $L_{123.6}$%
\hfil}%
\hbox to \myboxwidth{%
$4222$ (shared)%
\hfil}%
\box\matricesbox
\fi
}%
\hfill\discretionary{}{}{}%
\setbox\matricesbox=\hbox{%
{$\left[\!\llap{\phantom{%
\begingroup \smaller\smaller\smaller\begin{tabular}{@{}c@{}}%
\phantom{0}\\\phantom{0}\\\phantom{0}
\end{tabular}\endgroup%
}}\right.$}%
\begingroup \smaller\smaller\smaller\begin{tabular}{@{}c@{}}%
-1/19\\\phantom{0}\\\phantom{0}
\end{tabular}\endgroup%
\kern3pt%
\begingroup \smaller\smaller\smaller\begin{tabular}{@{}c@{}}%
\phantom{0}\\6/19\\-3/19
\end{tabular}\endgroup%
\kern3pt%
\begingroup \smaller\smaller\smaller\begin{tabular}{@{}c@{}}%
\phantom{0}\\-3/19\\30/19
\end{tabular}\endgroup%
{$\left.\llap{\phantom{%
\begingroup \smaller\smaller\smaller\begin{tabular}{@{}c@{}}%
\phantom{0}\\\phantom{0}\\\phantom{0}
\end{tabular}\endgroup%
}}\!\right]$}%
{$\left[\!\llap{\phantom{%
\begingroup \smaller\smaller\smaller\begin{tabular}{@{}c@{}}%
0\\0\\0
\end{tabular}\endgroup%
}}\right.$}%
\begingroup \smaller\smaller\smaller\begin{tabular}{@{}c@{}}%
6\\4\\-1
\end{tabular}\endgroup%
\kern3pt%
\begingroup \smaller\smaller\smaller\begin{tabular}{@{}c@{}}%
3\\-3\\-1
\end{tabular}\endgroup%
\kern3pt%
\begingroup \smaller\smaller\smaller\begin{tabular}{@{}c@{}}%
3\\-2\\1
\end{tabular}\endgroup%
\kern3pt%
\begingroup \smaller\smaller\smaller\begin{tabular}{@{}c@{}}%
2\\2\\1
\end{tabular}\endgroup%
{$\left.\llap{\phantom{%
\begingroup \smaller\smaller\smaller\begin{tabular}{@{}c@{}}%
0\\0\\0
\end{tabular}\endgroup%
}}\!\right]$}%
}%
\ifdim\wd\matricesbox>\halfwidth\myboxwidth=\hsize\else\myboxwidth=\halfwidth\fi
\vbox{%
\ifdim\myboxwidth=\hsize
\setbox\onelinebox=\hbox{%
\vbox{\hbox{%
$\Pi_{4,68}$ spans $L_{3.3}$%
}\hbox{%
$4222$ (shared)%
}%
}%
\hfill\copy\matricesbox
}%
\ifdim\wd\onelinebox>\myboxwidth
\hbox to \myboxwidth{%
$\Pi_{4,68}$ spans $L_{3.3}$%
\hfil
$4222$ (shared)%
}%
\box\matricesbox
\else
\hbox to \myboxwidth{%
\unhbox\onelinebox
}%
\fi
\else
\hbox to \myboxwidth{%
$\Pi_{4,68}$ spans $L_{3.3}$%
\hfil}%
\hbox to \myboxwidth{%
$4222$ (shared)%
\hfil}%
\box\matricesbox
\fi
}%
\hfill\discretionary{}{}{}%
\setbox\matricesbox=\hbox{%
{$\left[\!\llap{\phantom{%
\begingroup \smaller\smaller\smaller\begin{tabular}{@{}c@{}}%
\phantom{0}\\\phantom{0}\\\phantom{0}
\end{tabular}\endgroup%
}}\right.$}%
\begingroup \smaller\smaller\smaller\begin{tabular}{@{}c@{}}%
-1/184\\\phantom{0}\\\phantom{0}
\end{tabular}\endgroup%
\kern3pt%
\begingroup \smaller\smaller\smaller\begin{tabular}{@{}c@{}}%
\phantom{0}\\165/46\\-75/46
\end{tabular}\endgroup%
\kern3pt%
\begingroup \smaller\smaller\smaller\begin{tabular}{@{}c@{}}%
\phantom{0}\\-75/46\\285/46
\end{tabular}\endgroup%
{$\left.\llap{\phantom{%
\begingroup \smaller\smaller\smaller\begin{tabular}{@{}c@{}}%
\phantom{0}\\\phantom{0}\\\phantom{0}
\end{tabular}\endgroup%
}}\!\right]$}%
{$\left[\!\llap{\phantom{%
\begingroup \smaller\smaller\smaller\begin{tabular}{@{}c@{}}%
0\\0\\0
\end{tabular}\endgroup%
}}\right.$}%
\begingroup \smaller\smaller\smaller\begin{tabular}{@{}c@{}}%
30\\1\\-2
\end{tabular}\endgroup%
\kern3pt%
\begingroup \smaller\smaller\smaller\begin{tabular}{@{}c@{}}%
60\\-5\\-1
\end{tabular}\endgroup%
\kern3pt%
\begingroup \smaller\smaller\smaller\begin{tabular}{@{}c@{}}%
6\\0\\1
\end{tabular}\endgroup%
\kern3pt%
\begingroup \smaller\smaller\smaller\begin{tabular}{@{}c@{}}%
80\\6\\2
\end{tabular}\endgroup%
{$\left.\llap{\phantom{%
\begingroup \smaller\smaller\smaller\begin{tabular}{@{}c@{}}%
0\\0\\0
\end{tabular}\endgroup%
}}\!\right]$}%
}%
\ifdim\wd\matricesbox>\halfwidth\myboxwidth=\hsize\else\myboxwidth=\halfwidth\fi
\vbox{%
\ifdim\myboxwidth=\hsize
\setbox\onelinebox=\hbox{%
\vbox{\hbox{%
$\Pi_{4,69}$ spans $L_{19.10}$%
}\hbox{%
$4222$ (shared)%
}%
}%
\hfill\copy\matricesbox
}%
\ifdim\wd\onelinebox>\myboxwidth
\hbox to \myboxwidth{%
$\Pi_{4,69}$ spans $L_{19.10}$%
\hfil
$4222$ (shared)%
}%
\box\matricesbox
\else
\hbox to \myboxwidth{%
\unhbox\onelinebox
}%
\fi
\else
\hbox to \myboxwidth{%
$\Pi_{4,69}$ spans $L_{19.10}$%
\hfil}%
\hbox to \myboxwidth{%
$4222$ (shared)%
\hfil}%
\box\matricesbox
\fi
}%
\hfill\discretionary{}{}{}%
\setbox\matricesbox=\hbox{%
{$\left[\!\llap{\phantom{%
\begingroup \smaller\smaller\smaller\begin{tabular}{@{}c@{}}%
\phantom{0}\\\phantom{0}\\\phantom{0}
\end{tabular}\endgroup%
}}\right.$}%
\begingroup \smaller\smaller\smaller\begin{tabular}{@{}c@{}}%
-1/49\\\phantom{0}\\\phantom{0}
\end{tabular}\endgroup%
\kern3pt%
\begingroup \smaller\smaller\smaller\begin{tabular}{@{}c@{}}%
\phantom{0}\\30/49\\-6/49
\end{tabular}\endgroup%
\kern3pt%
\begingroup \smaller\smaller\smaller\begin{tabular}{@{}c@{}}%
\phantom{0}\\-6/49\\60/49
\end{tabular}\endgroup%
{$\left.\llap{\phantom{%
\begingroup \smaller\smaller\smaller\begin{tabular}{@{}c@{}}%
\phantom{0}\\\phantom{0}\\\phantom{0}
\end{tabular}\endgroup%
}}\!\right]$}%
{$\left[\!\llap{\phantom{%
\begingroup \smaller\smaller\smaller\begin{tabular}{@{}c@{}}%
0\\0\\0
\end{tabular}\endgroup%
}}\right.$}%
\begingroup \smaller\smaller\smaller\begin{tabular}{@{}c@{}}%
24\\6\\4
\end{tabular}\endgroup%
\kern3pt%
\begingroup \smaller\smaller\smaller\begin{tabular}{@{}c@{}}%
12\\2\\-3
\end{tabular}\endgroup%
\kern3pt%
\begingroup \smaller\smaller\smaller\begin{tabular}{@{}c@{}}%
3\\-2\\-1
\end{tabular}\endgroup%
\kern3pt%
\begingroup \smaller\smaller\smaller\begin{tabular}{@{}c@{}}%
2\\-1\\1
\end{tabular}\endgroup%
{$\left.\llap{\phantom{%
\begingroup \smaller\smaller\smaller\begin{tabular}{@{}c@{}}%
0\\0\\0
\end{tabular}\endgroup%
}}\!\right]$}%
}%
\ifdim\wd\matricesbox>\halfwidth\myboxwidth=\hsize\else\myboxwidth=\halfwidth\fi
\vbox{%
\ifdim\myboxwidth=\hsize
\setbox\onelinebox=\hbox{%
\vbox{\hbox{%
$\Pi_{4,70}$ spans $L_{123.8}$%
}\hbox{%
$4222$ (shared)%
}%
}%
\hfill\copy\matricesbox
}%
\ifdim\wd\onelinebox>\myboxwidth
\hbox to \myboxwidth{%
$\Pi_{4,70}$ spans $L_{123.8}$%
\hfil
$4222$ (shared)%
}%
\box\matricesbox
\else
\hbox to \myboxwidth{%
\unhbox\onelinebox
}%
\fi
\else
\hbox to \myboxwidth{%
$\Pi_{4,70}$ spans $L_{123.8}$%
\hfil}%
\hbox to \myboxwidth{%
$4222$ (shared)%
\hfil}%
\box\matricesbox
\fi
}%
\hfill\discretionary{}{}{}%
\setbox\matricesbox=\hbox{%
{$\left[\!\llap{\phantom{%
\begingroup \smaller\smaller\smaller\begin{tabular}{@{}c@{}}%
\phantom{0}\\\phantom{0}\\\phantom{0}
\end{tabular}\endgroup%
}}\right.$}%
\begingroup \smaller\smaller\smaller\begin{tabular}{@{}c@{}}%
-1/28\\\phantom{0}\\\phantom{0}
\end{tabular}\endgroup%
\kern3pt%
\begingroup \smaller\smaller\smaller\begin{tabular}{@{}c@{}}%
\phantom{0}\\15/7\\-3/7
\end{tabular}\endgroup%
\kern3pt%
\begingroup \smaller\smaller\smaller\begin{tabular}{@{}c@{}}%
\phantom{0}\\-3/7\\30/7
\end{tabular}\endgroup%
{$\left.\llap{\phantom{%
\begingroup \smaller\smaller\smaller\begin{tabular}{@{}c@{}}%
\phantom{0}\\\phantom{0}\\\phantom{0}
\end{tabular}\endgroup%
}}\!\right]$}%
{$\left[\!\llap{\phantom{%
\begingroup \smaller\smaller\smaller\begin{tabular}{@{}c@{}}%
0\\0\\0
\end{tabular}\endgroup%
}}\right.$}%
\begingroup \smaller\smaller\smaller\begin{tabular}{@{}c@{}}%
2\\1\\0
\end{tabular}\endgroup%
\kern3pt%
\begingroup \smaller\smaller\smaller\begin{tabular}{@{}c@{}}%
6\\-1\\1
\end{tabular}\endgroup%
\kern3pt%
\begingroup \smaller\smaller\smaller\begin{tabular}{@{}c@{}}%
14\\-3\\-1
\end{tabular}\endgroup%
\kern3pt%
\begingroup \smaller\smaller\smaller\begin{tabular}{@{}c@{}}%
12\\0\\-2
\end{tabular}\endgroup%
{$\left.\llap{\phantom{%
\begingroup \smaller\smaller\smaller\begin{tabular}{@{}c@{}}%
0\\0\\0
\end{tabular}\endgroup%
}}\!\right]$}%
}%
\ifdim\wd\matricesbox>\halfwidth\myboxwidth=\hsize\else\myboxwidth=\halfwidth\fi
\vbox{%
\ifdim\myboxwidth=\hsize
\setbox\onelinebox=\hbox{%
\vbox{\hbox{%
$\Pi_{4,71}$ spans $L_{22.5}$%
}\hbox{%
$6222$ (shared)%
}%
}%
\hfill\copy\matricesbox
}%
\ifdim\wd\onelinebox>\myboxwidth
\hbox to \myboxwidth{%
$\Pi_{4,71}$ spans $L_{22.5}$%
\hfil
$6222$ (shared)%
}%
\box\matricesbox
\else
\hbox to \myboxwidth{%
\unhbox\onelinebox
}%
\fi
\else
\hbox to \myboxwidth{%
$\Pi_{4,71}$ spans $L_{22.5}$%
\hfil}%
\hbox to \myboxwidth{%
$6222$ (shared)%
\hfil}%
\box\matricesbox
\fi
}%
\hfill\discretionary{}{}{}%
\setbox\matricesbox=\hbox{%
{$\left[\!\llap{\phantom{%
\begingroup \smaller\smaller\smaller\begin{tabular}{@{}c@{}}%
\phantom{0}\\\phantom{0}\\\phantom{0}
\end{tabular}\endgroup%
}}\right.$}%
\begingroup \smaller\smaller\smaller\begin{tabular}{@{}c@{}}%
-1/24\\\phantom{0}\\\phantom{0}
\end{tabular}\endgroup%
\kern3pt%
\begingroup \smaller\smaller\smaller\begin{tabular}{@{}c@{}}%
\phantom{0}\\13/6\\-1/3
\end{tabular}\endgroup%
\kern3pt%
\begingroup \smaller\smaller\smaller\begin{tabular}{@{}c@{}}%
\phantom{0}\\-1/3\\14/3
\end{tabular}\endgroup%
{$\left.\llap{\phantom{%
\begingroup \smaller\smaller\smaller\begin{tabular}{@{}c@{}}%
\phantom{0}\\\phantom{0}\\\phantom{0}
\end{tabular}\endgroup%
}}\!\right]$}%
{$\left[\!\llap{\phantom{%
\begingroup \smaller\smaller\smaller\begin{tabular}{@{}c@{}}%
0\\0\\0
\end{tabular}\endgroup%
}}\right.$}%
\begingroup \smaller\smaller\smaller\begin{tabular}{@{}c@{}}%
2\\1\\0
\end{tabular}\endgroup%
\kern3pt%
\begingroup \smaller\smaller\smaller\begin{tabular}{@{}c@{}}%
6\\-1\\1
\end{tabular}\endgroup%
\kern3pt%
\begingroup \smaller\smaller\smaller\begin{tabular}{@{}c@{}}%
20\\-4\\-1
\end{tabular}\endgroup%
\kern3pt%
\begingroup \smaller\smaller\smaller\begin{tabular}{@{}c@{}}%
4\\0\\-1
\end{tabular}\endgroup%
{$\left.\llap{\phantom{%
\begingroup \smaller\smaller\smaller\begin{tabular}{@{}c@{}}%
0\\0\\0
\end{tabular}\endgroup%
}}\!\right]$}%
}%
\ifdim\wd\matricesbox>\halfwidth\myboxwidth=\hsize\else\myboxwidth=\halfwidth\fi
\vbox{%
\ifdim\myboxwidth=\hsize
\setbox\onelinebox=\hbox{%
\vbox{\hbox{%
$\Pi_{4,72}$ spans $L_{16.3}$%
}\hbox{%
$6222$ (shared)%
}%
}%
\hfill\copy\matricesbox
}%
\ifdim\wd\onelinebox>\myboxwidth
\hbox to \myboxwidth{%
$\Pi_{4,72}$ spans $L_{16.3}$%
\hfil
$6222$ (shared)%
}%
\box\matricesbox
\else
\hbox to \myboxwidth{%
\unhbox\onelinebox
}%
\fi
\else
\hbox to \myboxwidth{%
$\Pi_{4,72}$ spans $L_{16.3}$%
\hfil}%
\hbox to \myboxwidth{%
$6222$ (shared)%
\hfil}%
\box\matricesbox
\fi
}%
\hfill\discretionary{}{}{}%
\setbox\matricesbox=\hbox{%
{$\left[\!\llap{\phantom{%
\begingroup \smaller\smaller\smaller\begin{tabular}{@{}c@{}}%
\phantom{0}\\\phantom{0}\\\phantom{0}
\end{tabular}\endgroup%
}}\right.$}%
\begingroup \smaller\smaller\smaller\begin{tabular}{@{}c@{}}%
-1/104\\\phantom{0}\\\phantom{0}
\end{tabular}\endgroup%
\kern3pt%
\begingroup \smaller\smaller\smaller\begin{tabular}{@{}c@{}}%
\phantom{0}\\165/26\\-15/26
\end{tabular}\endgroup%
\kern3pt%
\begingroup \smaller\smaller\smaller\begin{tabular}{@{}c@{}}%
\phantom{0}\\-15/26\\285/26
\end{tabular}\endgroup%
{$\left.\llap{\phantom{%
\begingroup \smaller\smaller\smaller\begin{tabular}{@{}c@{}}%
\phantom{0}\\\phantom{0}\\\phantom{0}
\end{tabular}\endgroup%
}}\!\right]$}%
{$\left[\!\llap{\phantom{%
\begingroup \smaller\smaller\smaller\begin{tabular}{@{}c@{}}%
0\\0\\0
\end{tabular}\endgroup%
}}\right.$}%
\begingroup \smaller\smaller\smaller\begin{tabular}{@{}c@{}}%
10\\0\\1
\end{tabular}\endgroup%
\kern3pt%
\begingroup \smaller\smaller\smaller\begin{tabular}{@{}c@{}}%
30\\2\\-1
\end{tabular}\endgroup%
\kern3pt%
\begingroup \smaller\smaller\smaller\begin{tabular}{@{}c@{}}%
160\\-2\\-6
\end{tabular}\endgroup%
\kern3pt%
\begingroup \smaller\smaller\smaller\begin{tabular}{@{}c@{}}%
6\\-1\\0
\end{tabular}\endgroup%
{$\left.\llap{\phantom{%
\begingroup \smaller\smaller\smaller\begin{tabular}{@{}c@{}}%
0\\0\\0
\end{tabular}\endgroup%
}}\!\right]$}%
}%
\ifdim\wd\matricesbox>\halfwidth\myboxwidth=\hsize\else\myboxwidth=\halfwidth\fi
\vbox{%
\ifdim\myboxwidth=\hsize
\setbox\onelinebox=\hbox{%
\vbox{\hbox{%
$\Pi_{4,73}$ spans $L_{31.13}$%
}\hbox{%
$6222$ (shared)%
}%
}%
\hfill\copy\matricesbox
}%
\ifdim\wd\onelinebox>\myboxwidth
\hbox to \myboxwidth{%
$\Pi_{4,73}$ spans $L_{31.13}$%
\hfil
$6222$ (shared)%
}%
\box\matricesbox
\else
\hbox to \myboxwidth{%
\unhbox\onelinebox
}%
\fi
\else
\hbox to \myboxwidth{%
$\Pi_{4,73}$ spans $L_{31.13}$%
\hfil}%
\hbox to \myboxwidth{%
$6222$ (shared)%
\hfil}%
\box\matricesbox
\fi
}%
\hfill\discretionary{}{}{}%
\setbox\matricesbox=\hbox{%
{$\left[\!\llap{\phantom{%
\begingroup \smaller\smaller\smaller\begin{tabular}{@{}c@{}}%
\phantom{0}\\\phantom{0}\\\phantom{0}
\end{tabular}\endgroup%
}}\right.$}%
\begingroup \smaller\smaller\smaller\begin{tabular}{@{}c@{}}%
-1/92\\\phantom{0}\\\phantom{0}
\end{tabular}\endgroup%
\kern3pt%
\begingroup \smaller\smaller\smaller\begin{tabular}{@{}c@{}}%
\phantom{0}\\35/23\\-7/23
\end{tabular}\endgroup%
\kern3pt%
\begingroup \smaller\smaller\smaller\begin{tabular}{@{}c@{}}%
\phantom{0}\\-7/23\\98/23
\end{tabular}\endgroup%
{$\left.\llap{\phantom{%
\begingroup \smaller\smaller\smaller\begin{tabular}{@{}c@{}}%
\phantom{0}\\\phantom{0}\\\phantom{0}
\end{tabular}\endgroup%
}}\!\right]$}%
{$\left[\!\llap{\phantom{%
\begingroup \smaller\smaller\smaller\begin{tabular}{@{}c@{}}%
0\\0\\0
\end{tabular}\endgroup%
}}\right.$}%
\begingroup \smaller\smaller\smaller\begin{tabular}{@{}c@{}}%
42\\-5\\2
\end{tabular}\endgroup%
\kern3pt%
\begingroup \smaller\smaller\smaller\begin{tabular}{@{}c@{}}%
14\\3\\1
\end{tabular}\endgroup%
\kern3pt%
\begingroup \smaller\smaller\smaller\begin{tabular}{@{}c@{}}%
6\\1\\-1
\end{tabular}\endgroup%
\kern3pt%
\begingroup \smaller\smaller\smaller\begin{tabular}{@{}c@{}}%
28\\-4\\-2
\end{tabular}\endgroup%
{$\left.\llap{\phantom{%
\begingroup \smaller\smaller\smaller\begin{tabular}{@{}c@{}}%
0\\0\\0
\end{tabular}\endgroup%
}}\!\right]$}%
}%
\ifdim\wd\matricesbox>\halfwidth\myboxwidth=\hsize\else\myboxwidth=\halfwidth\fi
\vbox{%
\ifdim\myboxwidth=\hsize
\setbox\onelinebox=\hbox{%
\vbox{\hbox{%
$\Pi_{4,74}$ spans $L_{22.8}$%
}\hbox{%
$6222$ (shared)%
}%
}%
\hfill\copy\matricesbox
}%
\ifdim\wd\onelinebox>\myboxwidth
\hbox to \myboxwidth{%
$\Pi_{4,74}$ spans $L_{22.8}$%
\hfil
$6222$ (shared)%
}%
\box\matricesbox
\else
\hbox to \myboxwidth{%
\unhbox\onelinebox
}%
\fi
\else
\hbox to \myboxwidth{%
$\Pi_{4,74}$ spans $L_{22.8}$%
\hfil}%
\hbox to \myboxwidth{%
$6222$ (shared)%
\hfil}%
\box\matricesbox
\fi
}%
\hfill\discretionary{}{}{}%
\setbox\matricesbox=\hbox{%
{$\left[\!\llap{\phantom{%
\begingroup \smaller\smaller\smaller\begin{tabular}{@{}c@{}}%
\phantom{0}\\\phantom{0}\\\phantom{0}
\end{tabular}\endgroup%
}}\right.$}%
\begingroup \smaller\smaller\smaller\begin{tabular}{@{}c@{}}%
-1/184\\\phantom{0}\\\phantom{0}
\end{tabular}\endgroup%
\kern3pt%
\begingroup \smaller\smaller\smaller\begin{tabular}{@{}c@{}}%
\phantom{0}\\165/46\\-45/46
\end{tabular}\endgroup%
\kern3pt%
\begingroup \smaller\smaller\smaller\begin{tabular}{@{}c@{}}%
\phantom{0}\\-45/46\\765/46
\end{tabular}\endgroup%
{$\left.\llap{\phantom{%
\begingroup \smaller\smaller\smaller\begin{tabular}{@{}c@{}}%
\phantom{0}\\\phantom{0}\\\phantom{0}
\end{tabular}\endgroup%
}}\!\right]$}%
{$\left[\!\llap{\phantom{%
\begingroup \smaller\smaller\smaller\begin{tabular}{@{}c@{}}%
0\\0\\0
\end{tabular}\endgroup%
}}\right.$}%
\begingroup \smaller\smaller\smaller\begin{tabular}{@{}c@{}}%
90\\-6\\-1
\end{tabular}\endgroup%
\kern3pt%
\begingroup \smaller\smaller\smaller\begin{tabular}{@{}c@{}}%
30\\2\\-1
\end{tabular}\endgroup%
\kern3pt%
\begingroup \smaller\smaller\smaller\begin{tabular}{@{}c@{}}%
36\\3\\1
\end{tabular}\endgroup%
\kern3pt%
\begingroup \smaller\smaller\smaller\begin{tabular}{@{}c@{}}%
20\\-1\\1
\end{tabular}\endgroup%
{$\left.\llap{\phantom{%
\begingroup \smaller\smaller\smaller\begin{tabular}{@{}c@{}}%
0\\0\\0
\end{tabular}\endgroup%
}}\!\right]$}%
}%
\ifdim\wd\matricesbox>\halfwidth\myboxwidth=\hsize\else\myboxwidth=\halfwidth\fi
\vbox{%
\ifdim\myboxwidth=\hsize
\setbox\onelinebox=\hbox{%
\vbox{\hbox{%
$\Pi_{4,75}$ spans $L_{16.20}$%
}\hbox{%
$6222$ (shared)%
}%
}%
\hfill\copy\matricesbox
}%
\ifdim\wd\onelinebox>\myboxwidth
\hbox to \myboxwidth{%
$\Pi_{4,75}$ spans $L_{16.20}$%
\hfil
$6222$ (shared)%
}%
\box\matricesbox
\else
\hbox to \myboxwidth{%
\unhbox\onelinebox
}%
\fi
\else
\hbox to \myboxwidth{%
$\Pi_{4,75}$ spans $L_{16.20}$%
\hfil}%
\hbox to \myboxwidth{%
$6222$ (shared)%
\hfil}%
\box\matricesbox
\fi
}%
\hfill\discretionary{}{}{}%
\setbox\matricesbox=\hbox{%
{$\left[\!\llap{\phantom{%
\begingroup \smaller\smaller\smaller\begin{tabular}{@{}c@{}}%
\phantom{0}\\\phantom{0}\\\phantom{0}
\end{tabular}\endgroup%
}}\right.$}%
\begingroup \smaller\smaller\smaller\begin{tabular}{@{}c@{}}%
-1/58\\\phantom{0}\\\phantom{0}
\end{tabular}\endgroup%
\kern3pt%
\begingroup \smaller\smaller\smaller\begin{tabular}{@{}c@{}}%
\phantom{0}\\40/29\\-15/29
\end{tabular}\endgroup%
\kern3pt%
\begingroup \smaller\smaller\smaller\begin{tabular}{@{}c@{}}%
\phantom{0}\\-15/29\\60/29
\end{tabular}\endgroup%
{$\left.\llap{\phantom{%
\begingroup \smaller\smaller\smaller\begin{tabular}{@{}c@{}}%
\phantom{0}\\\phantom{0}\\\phantom{0}
\end{tabular}\endgroup%
}}\!\right]$}%
{$\left[\!\llap{\phantom{%
\begingroup \smaller\smaller\smaller\begin{tabular}{@{}c@{}}%
0\\0\\0
\end{tabular}\endgroup%
}}\right.$}%
\begingroup \smaller\smaller\smaller\begin{tabular}{@{}c@{}}%
30\\6\\1
\end{tabular}\endgroup%
\kern3pt%
\begingroup \smaller\smaller\smaller\begin{tabular}{@{}c@{}}%
10\\-1\\2
\end{tabular}\endgroup%
\kern3pt%
\begingroup \smaller\smaller\smaller\begin{tabular}{@{}c@{}}%
30\\-6\\-2
\end{tabular}\endgroup%
\kern3pt%
\begingroup \smaller\smaller\smaller\begin{tabular}{@{}c@{}}%
2\\0\\-1
\end{tabular}\endgroup%
{$\left.\llap{\phantom{%
\begingroup \smaller\smaller\smaller\begin{tabular}{@{}c@{}}%
0\\0\\0
\end{tabular}\endgroup%
}}\!\right]$}%
}%
\ifdim\wd\matricesbox>\halfwidth\myboxwidth=\hsize\else\myboxwidth=\halfwidth\fi
\vbox{%
\ifdim\myboxwidth=\hsize
\setbox\onelinebox=\hbox{%
\vbox{\hbox{%
$\Pi_{4,76}$ spans $L_{31.5}$%
}\hbox{%
$6222$ (shared)%
}%
}%
\hfill\copy\matricesbox
}%
\ifdim\wd\onelinebox>\myboxwidth
\hbox to \myboxwidth{%
$\Pi_{4,76}$ spans $L_{31.5}$%
\hfil
$6222$ (shared)%
}%
\box\matricesbox
\else
\hbox to \myboxwidth{%
\unhbox\onelinebox
}%
\fi
\else
\hbox to \myboxwidth{%
$\Pi_{4,76}$ spans $L_{31.5}$%
\hfil}%
\hbox to \myboxwidth{%
$6222$ (shared)%
\hfil}%
\box\matricesbox
\fi
}%
\hfill\discretionary{}{}{}%
\setbox\matricesbox=\hbox{%
{$\left[\!\llap{\phantom{%
\begingroup \smaller\smaller\smaller\begin{tabular}{@{}c@{}}%
\phantom{0}\\\phantom{0}\\\phantom{0}
\end{tabular}\endgroup%
}}\right.$}%
\begingroup \smaller\smaller\smaller\begin{tabular}{@{}c@{}}%
-1/121\\\phantom{0}\\\phantom{0}
\end{tabular}\endgroup%
\kern3pt%
\begingroup \smaller\smaller\smaller\begin{tabular}{@{}c@{}}%
\phantom{0}\\210/121\\-15/121
\end{tabular}\endgroup%
\kern3pt%
\begingroup \smaller\smaller\smaller\begin{tabular}{@{}c@{}}%
\phantom{0}\\-15/121\\390/121
\end{tabular}\endgroup%
{$\left.\llap{\phantom{%
\begingroup \smaller\smaller\smaller\begin{tabular}{@{}c@{}}%
\phantom{0}\\\phantom{0}\\\phantom{0}
\end{tabular}\endgroup%
}}\!\right]$}%
{$\left[\!\llap{\phantom{%
\begingroup \smaller\smaller\smaller\begin{tabular}{@{}c@{}}%
0\\0\\0
\end{tabular}\endgroup%
}}\right.$}%
\begingroup \smaller\smaller\smaller\begin{tabular}{@{}c@{}}%
90\\7\\5
\end{tabular}\endgroup%
\kern3pt%
\begingroup \smaller\smaller\smaller\begin{tabular}{@{}c@{}}%
30\\2\\-3
\end{tabular}\endgroup%
\kern3pt%
\begingroup \smaller\smaller\smaller\begin{tabular}{@{}c@{}}%
9\\-2\\-1
\end{tabular}\endgroup%
\kern3pt%
\begingroup \smaller\smaller\smaller\begin{tabular}{@{}c@{}}%
5\\-1\\1
\end{tabular}\endgroup%
{$\left.\llap{\phantom{%
\begingroup \smaller\smaller\smaller\begin{tabular}{@{}c@{}}%
0\\0\\0
\end{tabular}\endgroup%
}}\!\right]$}%
}%
\ifdim\wd\matricesbox>\halfwidth\myboxwidth=\hsize\else\myboxwidth=\halfwidth\fi
\vbox{%
\ifdim\myboxwidth=\hsize
\setbox\onelinebox=\hbox{%
\vbox{\hbox{%
$\Pi_{4,77}$ spans $L_{16.13}$%
}\hbox{%
$6222$ (shared)%
}%
}%
\hfill\copy\matricesbox
}%
\ifdim\wd\onelinebox>\myboxwidth
\hbox to \myboxwidth{%
$\Pi_{4,77}$ spans $L_{16.13}$%
\hfil
$6222$ (shared)%
}%
\box\matricesbox
\else
\hbox to \myboxwidth{%
\unhbox\onelinebox
}%
\fi
\else
\hbox to \myboxwidth{%
$\Pi_{4,77}$ spans $L_{16.13}$%
\hfil}%
\hbox to \myboxwidth{%
$6222$ (shared)%
\hfil}%
\box\matricesbox
\fi
}%
\hfill\discretionary{}{}{}%
\setbox\matricesbox=\hbox{%
{$\left[\!\llap{\phantom{%
\begingroup \smaller\smaller\smaller\begin{tabular}{@{}c@{}}%
\phantom{0}\\\phantom{0}\\\phantom{0}
\end{tabular}\endgroup%
}}\right.$}%
\begingroup \smaller\smaller\smaller\begin{tabular}{@{}c@{}}%
-1/14\\\phantom{0}\\\phantom{0}
\end{tabular}\endgroup%
\kern3pt%
\begingroup \smaller\smaller\smaller\begin{tabular}{@{}c@{}}%
\phantom{0}\\15/14\\-3/7
\end{tabular}\endgroup%
\kern3pt%
\begingroup \smaller\smaller\smaller\begin{tabular}{@{}c@{}}%
\phantom{0}\\-3/7\\18/7
\end{tabular}\endgroup%
{$\left.\llap{\phantom{%
\begingroup \smaller\smaller\smaller\begin{tabular}{@{}c@{}}%
\phantom{0}\\\phantom{0}\\\phantom{0}
\end{tabular}\endgroup%
}}\!\right]$}%
{$\left[\!\llap{\phantom{%
\begingroup \smaller\smaller\smaller\begin{tabular}{@{}c@{}}%
0\\0\\0
\end{tabular}\endgroup%
}}\right.$}%
\begingroup \smaller\smaller\smaller\begin{tabular}{@{}c@{}}%
1\\1\\0
\end{tabular}\endgroup%
\kern3pt%
\begingroup \smaller\smaller\smaller\begin{tabular}{@{}c@{}}%
4\\-2\\-1
\end{tabular}\endgroup%
\kern3pt%
\begingroup \smaller\smaller\smaller\begin{tabular}{@{}c@{}}%
6\\-2\\1
\end{tabular}\endgroup%
\kern3pt%
\begingroup \smaller\smaller\smaller\begin{tabular}{@{}c@{}}%
12\\2\\3
\end{tabular}\endgroup%
{$\left.\llap{\phantom{%
\begingroup \smaller\smaller\smaller\begin{tabular}{@{}c@{}}%
0\\0\\0
\end{tabular}\endgroup%
}}\!\right]$}%
}%
\ifdim\wd\matricesbox>\halfwidth\myboxwidth=\hsize\else\myboxwidth=\halfwidth\fi
\vbox{%
\ifdim\myboxwidth=\hsize
\setbox\onelinebox=\hbox{%
\vbox{\hbox{%
$\Pi_{4,78}$ spans $L_{4.17}$%
}\hbox{%
$\infty222$ (shared)%
}%
}%
\hfill\copy\matricesbox
}%
\ifdim\wd\onelinebox>\myboxwidth
\hbox to \myboxwidth{%
$\Pi_{4,78}$ spans $L_{4.17}$%
\hfil
$\infty222$ (shared)%
}%
\box\matricesbox
\else
\hbox to \myboxwidth{%
\unhbox\onelinebox
}%
\fi
\else
\hbox to \myboxwidth{%
$\Pi_{4,78}$ spans $L_{4.17}$%
\hfil}%
\hbox to \myboxwidth{%
$\infty222$ (shared)%
\hfil}%
\box\matricesbox
\fi
}%
\hfill\discretionary{}{}{}%
\setbox\matricesbox=\hbox{%
{$\left[\!\llap{\phantom{%
\begingroup \smaller\smaller\smaller\begin{tabular}{@{}c@{}}%
\phantom{0}\\\phantom{0}\\\phantom{0}
\end{tabular}\endgroup%
}}\right.$}%
\begingroup \smaller\smaller\smaller\begin{tabular}{@{}c@{}}%
-1/24\\\phantom{0}\\\phantom{0}
\end{tabular}\endgroup%
\kern3pt%
\begingroup \smaller\smaller\smaller\begin{tabular}{@{}c@{}}%
\phantom{0}\\13/6\\-1
\end{tabular}\endgroup%
\kern3pt%
\begingroup \smaller\smaller\smaller\begin{tabular}{@{}c@{}}%
\phantom{0}\\-1\\6
\end{tabular}\endgroup%
{$\left.\llap{\phantom{%
\begingroup \smaller\smaller\smaller\begin{tabular}{@{}c@{}}%
\phantom{0}\\\phantom{0}\\\phantom{0}
\end{tabular}\endgroup%
}}\!\right]$}%
{$\left[\!\llap{\phantom{%
\begingroup \smaller\smaller\smaller\begin{tabular}{@{}c@{}}%
0\\0\\0
\end{tabular}\endgroup%
}}\right.$}%
\begingroup \smaller\smaller\smaller\begin{tabular}{@{}c@{}}%
2\\1\\0
\end{tabular}\endgroup%
\kern3pt%
\begingroup \smaller\smaller\smaller\begin{tabular}{@{}c@{}}%
8\\-2\\-1
\end{tabular}\endgroup%
\kern3pt%
\begingroup \smaller\smaller\smaller\begin{tabular}{@{}c@{}}%
18\\-3\\1
\end{tabular}\endgroup%
\kern3pt%
\begingroup \smaller\smaller\smaller\begin{tabular}{@{}c@{}}%
72\\6\\7
\end{tabular}\endgroup%
{$\left.\llap{\phantom{%
\begingroup \smaller\smaller\smaller\begin{tabular}{@{}c@{}}%
0\\0\\0
\end{tabular}\endgroup%
}}\!\right]$}%
}%
\ifdim\wd\matricesbox>\halfwidth\myboxwidth=\hsize\else\myboxwidth=\halfwidth\fi
\vbox{%
\ifdim\myboxwidth=\hsize
\setbox\onelinebox=\hbox{%
\vbox{\hbox{%
$\Pi_{4,79}=\hbox{GN}_{9}$ spans $L_{148.13}$%
}\hbox{%
$\infty2\infty2$ (shared)%
}%
}%
\hfill\copy\matricesbox
}%
\ifdim\wd\onelinebox>\myboxwidth
\hbox to \myboxwidth{%
$\Pi_{4,79}=\hbox{GN}_{9}$ spans $L_{148.13}$%
\hfil
$\infty2\infty2$ (shared)%
}%
\box\matricesbox
\else
\hbox to \myboxwidth{%
\unhbox\onelinebox
}%
\fi
\else
\hbox to \myboxwidth{%
$\Pi_{4,79}=\hbox{GN}_{9}$ spans $L_{148.13}$%
\hfil}%
\hbox to \myboxwidth{%
$\infty2\infty2$ (shared)%
\hfil}%
\box\matricesbox
\fi
}%
\hfill\discretionary{}{}{}%
\setbox\matricesbox=\hbox{%
{$\left[\!\llap{\phantom{%
\begingroup \smaller\smaller\smaller\begin{tabular}{@{}c@{}}%
\phantom{0}\\\phantom{0}\\\phantom{0}
\end{tabular}\endgroup%
}}\right.$}%
\begingroup \smaller\smaller\smaller\begin{tabular}{@{}c@{}}%
-1/16\\\phantom{0}\\\phantom{0}
\end{tabular}\endgroup%
\kern3pt%
\begingroup \smaller\smaller\smaller\begin{tabular}{@{}c@{}}%
\phantom{0}\\9/16\\-3/16
\end{tabular}\endgroup%
\kern3pt%
\begingroup \smaller\smaller\smaller\begin{tabular}{@{}c@{}}%
\phantom{0}\\-3/16\\33/16
\end{tabular}\endgroup%
{$\left.\llap{\phantom{%
\begingroup \smaller\smaller\smaller\begin{tabular}{@{}c@{}}%
\phantom{0}\\\phantom{0}\\\phantom{0}
\end{tabular}\endgroup%
}}\!\right]$}%
{$\left[\!\llap{\phantom{%
\begingroup \smaller\smaller\smaller\begin{tabular}{@{}c@{}}%
0\\0\\0
\end{tabular}\endgroup%
}}\right.$}%
\begingroup \smaller\smaller\smaller\begin{tabular}{@{}c@{}}%
2\\1\\1
\end{tabular}\endgroup%
\kern3pt%
\begingroup \smaller\smaller\smaller\begin{tabular}{@{}c@{}}%
8\\2\\-2
\end{tabular}\endgroup%
\kern3pt%
\begingroup \smaller\smaller\smaller\begin{tabular}{@{}c@{}}%
3\\-2\\-1
\end{tabular}\endgroup%
\kern3pt%
\begingroup \smaller\smaller\smaller\begin{tabular}{@{}c@{}}%
6\\-3\\1
\end{tabular}\endgroup%
{$\left.\llap{\phantom{%
\begingroup \smaller\smaller\smaller\begin{tabular}{@{}c@{}}%
0\\0\\0
\end{tabular}\endgroup%
}}\!\right]$}%
}%
\ifdim\wd\matricesbox>\halfwidth\myboxwidth=\hsize\else\myboxwidth=\halfwidth\fi
\vbox{%
\ifdim\myboxwidth=\hsize
\setbox\onelinebox=\hbox{%
\vbox{\hbox{%
$\Pi_{4,80}$ spans $L_{4.15}$%
}\hbox{%
$\infty222$ (shared)%
}%
}%
\hfill\copy\matricesbox
}%
\ifdim\wd\onelinebox>\myboxwidth
\hbox to \myboxwidth{%
$\Pi_{4,80}$ spans $L_{4.15}$%
\hfil
$\infty222$ (shared)%
}%
\box\matricesbox
\else
\hbox to \myboxwidth{%
\unhbox\onelinebox
}%
\fi
\else
\hbox to \myboxwidth{%
$\Pi_{4,80}$ spans $L_{4.15}$%
\hfil}%
\hbox to \myboxwidth{%
$\infty222$ (shared)%
\hfil}%
\box\matricesbox
\fi
}%
\hfill\discretionary{}{}{}%
\setbox\matricesbox=\hbox{%
{$\left[\!\llap{\phantom{%
\begingroup \smaller\smaller\smaller\begin{tabular}{@{}c@{}}%
\phantom{0}\\\phantom{0}\\\phantom{0}
\end{tabular}\endgroup%
}}\right.$}%
\begingroup \smaller\smaller\smaller\begin{tabular}{@{}c@{}}%
-1/13\\\phantom{0}\\\phantom{0}
\end{tabular}\endgroup%
\kern3pt%
\begingroup \smaller\smaller\smaller\begin{tabular}{@{}c@{}}%
\phantom{0}\\10/13\\-4/13
\end{tabular}\endgroup%
\kern3pt%
\begingroup \smaller\smaller\smaller\begin{tabular}{@{}c@{}}%
\phantom{0}\\-4/13\\12/13
\end{tabular}\endgroup%
{$\left.\llap{\phantom{%
\begingroup \smaller\smaller\smaller\begin{tabular}{@{}c@{}}%
\phantom{0}\\\phantom{0}\\\phantom{0}
\end{tabular}\endgroup%
}}\!\right]$}%
{$\left[\!\llap{\phantom{%
\begingroup \smaller\smaller\smaller\begin{tabular}{@{}c@{}}%
0\\0\\0
\end{tabular}\endgroup%
}}\right.$}%
\begingroup \smaller\smaller\smaller\begin{tabular}{@{}c@{}}%
2\\-1\\1
\end{tabular}\endgroup%
\kern3pt%
\begingroup \smaller\smaller\smaller\begin{tabular}{@{}c@{}}%
8\\-2\\-4
\end{tabular}\endgroup%
\kern3pt%
\begingroup \smaller\smaller\smaller\begin{tabular}{@{}c@{}}%
4\\2\\-1
\end{tabular}\endgroup%
\kern3pt%
\begingroup \smaller\smaller\smaller\begin{tabular}{@{}c@{}}%
1\\1\\1
\end{tabular}\endgroup%
{$\left.\llap{\phantom{%
\begingroup \smaller\smaller\smaller\begin{tabular}{@{}c@{}}%
0\\0\\0
\end{tabular}\endgroup%
}}\!\right]$}%
}%
\ifdim\wd\matricesbox>\halfwidth\myboxwidth=\hsize\else\myboxwidth=\halfwidth\fi
\vbox{%
\ifdim\myboxwidth=\hsize
\setbox\onelinebox=\hbox{%
\vbox{\hbox{%
$\Pi_{4,81}$ spans $L_{141.7}$%
}\hbox{%
$\infty222$ (shared)%
}%
}%
\hfill\copy\matricesbox
}%
\ifdim\wd\onelinebox>\myboxwidth
\hbox to \myboxwidth{%
$\Pi_{4,81}$ spans $L_{141.7}$%
\hfil
$\infty222$ (shared)%
}%
\box\matricesbox
\else
\hbox to \myboxwidth{%
\unhbox\onelinebox
}%
\fi
\else
\hbox to \myboxwidth{%
$\Pi_{4,81}$ spans $L_{141.7}$%
\hfil}%
\hbox to \myboxwidth{%
$\infty222$ (shared)%
\hfil}%
\box\matricesbox
\fi
}%
\hfill\discretionary{}{}{}%
\setbox\matricesbox=\hbox{%
{$\left[\!\llap{\phantom{%
\begingroup \smaller\smaller\smaller\begin{tabular}{@{}c@{}}%
\phantom{0}\\\phantom{0}\\\phantom{0}
\end{tabular}\endgroup%
}}\right.$}%
\begingroup \smaller\smaller\smaller\begin{tabular}{@{}c@{}}%
-1/18\\\phantom{0}\\\phantom{0}
\end{tabular}\endgroup%
\kern3pt%
\begingroup \smaller\smaller\smaller\begin{tabular}{@{}c@{}}%
\phantom{0}\\7/18\\-1/18
\end{tabular}\endgroup%
\kern3pt%
\begingroup \smaller\smaller\smaller\begin{tabular}{@{}c@{}}%
\phantom{0}\\-1/18\\31/18
\end{tabular}\endgroup%
{$\left.\llap{\phantom{%
\begingroup \smaller\smaller\smaller\begin{tabular}{@{}c@{}}%
\phantom{0}\\\phantom{0}\\\phantom{0}
\end{tabular}\endgroup%
}}\!\right]$}%
{$\left[\!\llap{\phantom{%
\begingroup \smaller\smaller\smaller\begin{tabular}{@{}c@{}}%
0\\0\\0
\end{tabular}\endgroup%
}}\right.$}%
\begingroup \smaller\smaller\smaller\begin{tabular}{@{}c@{}}%
3\\-2\\1
\end{tabular}\endgroup%
\kern3pt%
\begingroup \smaller\smaller\smaller\begin{tabular}{@{}c@{}}%
12\\7\\1
\end{tabular}\endgroup%
\kern3pt%
\begingroup \smaller\smaller\smaller\begin{tabular}{@{}c@{}}%
2\\1\\-1
\end{tabular}\endgroup%
\kern3pt%
\begingroup \smaller\smaller\smaller\begin{tabular}{@{}c@{}}%
4\\-3\\-1
\end{tabular}\endgroup%
{$\left.\llap{\phantom{%
\begingroup \smaller\smaller\smaller\begin{tabular}{@{}c@{}}%
0\\0\\0
\end{tabular}\endgroup%
}}\!\right]$}%
}%
\ifdim\wd\matricesbox>\halfwidth\myboxwidth=\hsize\else\myboxwidth=\halfwidth\fi
\vbox{%
\ifdim\myboxwidth=\hsize
\setbox\onelinebox=\hbox{%
\vbox{\hbox{%
$\Pi_{4,82}$ spans $L_{4.6}$%
}\hbox{%
$\infty222$ (shared)%
}%
}%
\hfill\copy\matricesbox
}%
\ifdim\wd\onelinebox>\myboxwidth
\hbox to \myboxwidth{%
$\Pi_{4,82}$ spans $L_{4.6}$%
\hfil
$\infty222$ (shared)%
}%
\box\matricesbox
\else
\hbox to \myboxwidth{%
\unhbox\onelinebox
}%
\fi
\else
\hbox to \myboxwidth{%
$\Pi_{4,82}$ spans $L_{4.6}$%
\hfil}%
\hbox to \myboxwidth{%
$\infty222$ (shared)%
\hfil}%
\box\matricesbox
\fi
}%
\hfill\discretionary{}{}{}%
\setbox\matricesbox=\hbox{%
{$\left[\!\llap{\phantom{%
\begingroup \smaller\smaller\smaller\begin{tabular}{@{}c@{}}%
\phantom{0}\\\phantom{0}\\\phantom{0}
\end{tabular}\endgroup%
}}\right.$}%
\begingroup \smaller\smaller\smaller\begin{tabular}{@{}c@{}}%
-1/56\\\phantom{0}\\\phantom{0}
\end{tabular}\endgroup%
\kern3pt%
\begingroup \smaller\smaller\smaller\begin{tabular}{@{}c@{}}%
\phantom{0}\\3/2\\-1/2
\end{tabular}\endgroup%
\kern3pt%
\begingroup \smaller\smaller\smaller\begin{tabular}{@{}c@{}}%
\phantom{0}\\-1/2\\29/14
\end{tabular}\endgroup%
{$\left.\llap{\phantom{%
\begingroup \smaller\smaller\smaller\begin{tabular}{@{}c@{}}%
\phantom{0}\\\phantom{0}\\\phantom{0}
\end{tabular}\endgroup%
}}\!\right]$}%
{$\left[\!\llap{\phantom{%
\begingroup \smaller\smaller\smaller\begin{tabular}{@{}c@{}}%
0\\0\\0
\end{tabular}\endgroup%
}}\right.$}%
\begingroup \smaller\smaller\smaller\begin{tabular}{@{}c@{}}%
10\\-1\\2
\end{tabular}\endgroup%
\kern3pt%
\begingroup \smaller\smaller\smaller\begin{tabular}{@{}c@{}}%
40\\7\\1
\end{tabular}\endgroup%
\kern3pt%
\begingroup \smaller\smaller\smaller\begin{tabular}{@{}c@{}}%
2\\0\\-1
\end{tabular}\endgroup%
\kern3pt%
\begingroup \smaller\smaller\smaller\begin{tabular}{@{}c@{}}%
32\\-6\\-2
\end{tabular}\endgroup%
{$\left.\llap{\phantom{%
\begingroup \smaller\smaller\smaller\begin{tabular}{@{}c@{}}%
0\\0\\0
\end{tabular}\endgroup%
}}\!\right]$}%
}%
\ifdim\wd\matricesbox>\halfwidth\myboxwidth=\hsize\else\myboxwidth=\halfwidth\fi
\vbox{%
\ifdim\myboxwidth=\hsize
\setbox\onelinebox=\hbox{%
\vbox{\hbox{%
$\Pi_{4,83}$ spans $L_{10.3}$%
}\hbox{%
$\infty222$ (shared)%
}%
}%
\hfill\copy\matricesbox
}%
\ifdim\wd\onelinebox>\myboxwidth
\hbox to \myboxwidth{%
$\Pi_{4,83}$ spans $L_{10.3}$%
\hfil
$\infty222$ (shared)%
}%
\box\matricesbox
\else
\hbox to \myboxwidth{%
\unhbox\onelinebox
}%
\fi
\else
\hbox to \myboxwidth{%
$\Pi_{4,83}$ spans $L_{10.3}$%
\hfil}%
\hbox to \myboxwidth{%
$\infty222$ (shared)%
\hfil}%
\box\matricesbox
\fi
}%
\hfill\discretionary{}{}{}%
\setbox\matricesbox=\hbox{%
{$\left[\!\llap{\phantom{%
\begingroup \smaller\smaller\smaller\begin{tabular}{@{}c@{}}%
\phantom{0}\\\phantom{0}\\\phantom{0}
\end{tabular}\endgroup%
}}\right.$}%
\begingroup \smaller\smaller\smaller\begin{tabular}{@{}c@{}}%
-1/24\\\phantom{0}\\\phantom{0}
\end{tabular}\endgroup%
\kern3pt%
\begingroup \smaller\smaller\smaller\begin{tabular}{@{}c@{}}%
\phantom{0}\\3/8\\\phantom{0}
\end{tabular}\endgroup%
\kern3pt%
\begingroup \smaller\smaller\smaller\begin{tabular}{@{}c@{}}%
\phantom{0}\\\phantom{0}\\2/3
\end{tabular}\endgroup%
{$\left.\llap{\phantom{%
\begingroup \smaller\smaller\smaller\begin{tabular}{@{}c@{}}%
\phantom{0}\\\phantom{0}\\\phantom{0}
\end{tabular}\endgroup%
}}\!\right]$}%
{$\left[\!\llap{\phantom{%
\begingroup \smaller\smaller\smaller\begin{tabular}{@{}c@{}}%
0\\0\\0
\end{tabular}\endgroup%
}}\right.$}%
\begingroup \smaller\smaller\smaller\begin{tabular}{@{}c@{}}%
6\\-2\\-3
\end{tabular}\endgroup%
\kern3pt%
\begingroup \smaller\smaller\smaller\begin{tabular}{@{}c@{}}%
24\\-8\\6
\end{tabular}\endgroup%
\kern3pt%
\begingroup \smaller\smaller\smaller\begin{tabular}{@{}c@{}}%
1\\1\\1
\end{tabular}\endgroup%
\kern3pt%
\begingroup \smaller\smaller\smaller\begin{tabular}{@{}c@{}}%
2\\2\\-1
\end{tabular}\endgroup%
{$\left.\llap{\phantom{%
\begingroup \smaller\smaller\smaller\begin{tabular}{@{}c@{}}%
0\\0\\0
\end{tabular}\endgroup%
}}\!\right]$}%
}%
\ifdim\wd\matricesbox>\halfwidth\myboxwidth=\hsize\else\myboxwidth=\halfwidth\fi
\vbox{%
\ifdim\myboxwidth=\hsize
\setbox\onelinebox=\hbox{%
\vbox{\hbox{%
$\Pi_{4,84}$ spans $L_{4.5}$%
}\hbox{%
$\infty222$ (shared)%
}%
}%
\hfill\copy\matricesbox
}%
\ifdim\wd\onelinebox>\myboxwidth
\hbox to \myboxwidth{%
$\Pi_{4,84}$ spans $L_{4.5}$%
\hfil
$\infty222$ (shared)%
}%
\box\matricesbox
\else
\hbox to \myboxwidth{%
\unhbox\onelinebox
}%
\fi
\else
\hbox to \myboxwidth{%
$\Pi_{4,84}$ spans $L_{4.5}$%
\hfil}%
\hbox to \myboxwidth{%
$\infty222$ (shared)%
\hfil}%
\box\matricesbox
\fi
}%
\hfill\discretionary{}{}{}%
\setbox\matricesbox=\hbox{%
{$\left[\!\llap{\phantom{%
\begingroup \smaller\smaller\smaller\begin{tabular}{@{}c@{}}%
\phantom{0}\\\phantom{0}\\\phantom{0}
\end{tabular}\endgroup%
}}\right.$}%
\begingroup \smaller\smaller\smaller\begin{tabular}{@{}c@{}}%
-1/13\\\phantom{0}\\\phantom{0}
\end{tabular}\endgroup%
\kern3pt%
\begingroup \smaller\smaller\smaller\begin{tabular}{@{}c@{}}%
\phantom{0}\\12/13\\-3/13
\end{tabular}\endgroup%
\kern3pt%
\begingroup \smaller\smaller\smaller\begin{tabular}{@{}c@{}}%
\phantom{0}\\-3/13\\30/13
\end{tabular}\endgroup%
{$\left.\llap{\phantom{%
\begingroup \smaller\smaller\smaller\begin{tabular}{@{}c@{}}%
\phantom{0}\\\phantom{0}\\\phantom{0}
\end{tabular}\endgroup%
}}\!\right]$}%
{$\left[\!\llap{\phantom{%
\begingroup \smaller\smaller\smaller\begin{tabular}{@{}c@{}}%
0\\0\\0
\end{tabular}\endgroup%
}}\right.$}%
\begingroup \smaller\smaller\smaller\begin{tabular}{@{}c@{}}%
3\\1\\-1
\end{tabular}\endgroup%
\kern3pt%
\begingroup \smaller\smaller\smaller\begin{tabular}{@{}c@{}}%
3\\-2\\0
\end{tabular}\endgroup%
\kern3pt%
\begingroup \smaller\smaller\smaller\begin{tabular}{@{}c@{}}%
2\\0\\1
\end{tabular}\endgroup%
\kern3pt%
\begingroup \smaller\smaller\smaller\begin{tabular}{@{}c@{}}%
9\\4\\1
\end{tabular}\endgroup%
{$\left.\llap{\phantom{%
\begingroup \smaller\smaller\smaller\begin{tabular}{@{}c@{}}%
0\\0\\0
\end{tabular}\endgroup%
}}\!\right]$}%
}%
\ifdim\wd\matricesbox>\halfwidth\myboxwidth=\hsize\else\myboxwidth=\halfwidth\fi
\vbox{%
\ifdim\myboxwidth=\hsize
\setbox\onelinebox=\hbox{%
\vbox{\hbox{%
$\Pi_{4,85}$ spans $L_{4.16}$%
}\hbox{%
$\infty222$ (shared)%
}%
}%
\hfill\copy\matricesbox
}%
\ifdim\wd\onelinebox>\myboxwidth
\hbox to \myboxwidth{%
$\Pi_{4,85}$ spans $L_{4.16}$%
\hfil
$\infty222$ (shared)%
}%
\box\matricesbox
\else
\hbox to \myboxwidth{%
\unhbox\onelinebox
}%
\fi
\else
\hbox to \myboxwidth{%
$\Pi_{4,85}$ spans $L_{4.16}$%
\hfil}%
\hbox to \myboxwidth{%
$\infty222$ (shared)%
\hfil}%
\box\matricesbox
\fi
}%
\hfill\discretionary{}{}{}%
\setbox\matricesbox=\hbox{%
{$\left[\!\llap{\phantom{%
\begingroup \smaller\smaller\smaller\begin{tabular}{@{}c@{}}%
\phantom{0}\\\phantom{0}\\\phantom{0}
\end{tabular}\endgroup%
}}\right.$}%
\begingroup \smaller\smaller\smaller\begin{tabular}{@{}c@{}}%
-1/38\\\phantom{0}\\\phantom{0}
\end{tabular}\endgroup%
\kern3pt%
\begingroup \smaller\smaller\smaller\begin{tabular}{@{}c@{}}%
\phantom{0}\\40/19\\-5/19
\end{tabular}\endgroup%
\kern3pt%
\begingroup \smaller\smaller\smaller\begin{tabular}{@{}c@{}}%
\phantom{0}\\-5/19\\60/19
\end{tabular}\endgroup%
{$\left.\llap{\phantom{%
\begingroup \smaller\smaller\smaller\begin{tabular}{@{}c@{}}%
\phantom{0}\\\phantom{0}\\\phantom{0}
\end{tabular}\endgroup%
}}\!\right]$}%
{$\left[\!\llap{\phantom{%
\begingroup \smaller\smaller\smaller\begin{tabular}{@{}c@{}}%
0\\0\\0
\end{tabular}\endgroup%
}}\right.$}%
\begingroup \smaller\smaller\smaller\begin{tabular}{@{}c@{}}%
10\\2\\-1
\end{tabular}\endgroup%
\kern3pt%
\begingroup \smaller\smaller\smaller\begin{tabular}{@{}c@{}}%
10\\0\\2
\end{tabular}\endgroup%
\kern3pt%
\begingroup \smaller\smaller\smaller\begin{tabular}{@{}c@{}}%
2\\-1\\0
\end{tabular}\endgroup%
\kern3pt%
\begingroup \smaller\smaller\smaller\begin{tabular}{@{}c@{}}%
50\\-2\\-6
\end{tabular}\endgroup%
{$\left.\llap{\phantom{%
\begingroup \smaller\smaller\smaller\begin{tabular}{@{}c@{}}%
0\\0\\0
\end{tabular}\endgroup%
}}\!\right]$}%
}%
\ifdim\wd\matricesbox>\halfwidth\myboxwidth=\hsize\else\myboxwidth=\halfwidth\fi
\vbox{%
\ifdim\myboxwidth=\hsize
\setbox\onelinebox=\hbox{%
\vbox{\hbox{%
$\Pi_{4,86}$ spans $L_{10.9}$%
}\hbox{%
$\infty222$ (shared)%
}%
}%
\hfill\copy\matricesbox
}%
\ifdim\wd\onelinebox>\myboxwidth
\hbox to \myboxwidth{%
$\Pi_{4,86}$ spans $L_{10.9}$%
\hfil
$\infty222$ (shared)%
}%
\box\matricesbox
\else
\hbox to \myboxwidth{%
\unhbox\onelinebox
}%
\fi
\else
\hbox to \myboxwidth{%
$\Pi_{4,86}$ spans $L_{10.9}$%
\hfil}%
\hbox to \myboxwidth{%
$\infty222$ (shared)%
\hfil}%
\box\matricesbox
\fi
}%
\hfill\discretionary{}{}{}%
\setbox\matricesbox=\hbox{%
{$\left[\!\llap{\phantom{%
\begingroup \smaller\smaller\smaller\begin{tabular}{@{}c@{}}%
\phantom{0}\\\phantom{0}\\\phantom{0}
\end{tabular}\endgroup%
}}\right.$}%
\begingroup \smaller\smaller\smaller\begin{tabular}{@{}c@{}}%
-1/18\\\phantom{0}\\\phantom{0}
\end{tabular}\endgroup%
\kern3pt%
\begingroup \smaller\smaller\smaller\begin{tabular}{@{}c@{}}%
\phantom{0}\\7/18\\-1/9
\end{tabular}\endgroup%
\kern3pt%
\begingroup \smaller\smaller\smaller\begin{tabular}{@{}c@{}}%
\phantom{0}\\-1/9\\8/9
\end{tabular}\endgroup%
{$\left.\llap{\phantom{%
\begingroup \smaller\smaller\smaller\begin{tabular}{@{}c@{}}%
\phantom{0}\\\phantom{0}\\\phantom{0}
\end{tabular}\endgroup%
}}\!\right]$}%
{$\left[\!\llap{\phantom{%
\begingroup \smaller\smaller\smaller\begin{tabular}{@{}c@{}}%
0\\0\\0
\end{tabular}\endgroup%
}}\right.$}%
\begingroup \smaller\smaller\smaller\begin{tabular}{@{}c@{}}%
6\\4\\2
\end{tabular}\endgroup%
\kern3pt%
\begingroup \smaller\smaller\smaller\begin{tabular}{@{}c@{}}%
6\\-4\\1
\end{tabular}\endgroup%
\kern3pt%
\begingroup \smaller\smaller\smaller\begin{tabular}{@{}c@{}}%
1\\-1\\-1
\end{tabular}\endgroup%
\kern3pt%
\begingroup \smaller\smaller\smaller\begin{tabular}{@{}c@{}}%
8\\4\\-2
\end{tabular}\endgroup%
{$\left.\llap{\phantom{%
\begingroup \smaller\smaller\smaller\begin{tabular}{@{}c@{}}%
0\\0\\0
\end{tabular}\endgroup%
}}\!\right]$}%
}%
\ifdim\wd\matricesbox>\halfwidth\myboxwidth=\hsize\else\myboxwidth=\halfwidth\fi
\vbox{%
\ifdim\myboxwidth=\hsize
\setbox\onelinebox=\hbox{%
\vbox{\hbox{%
$\Pi_{4,87}$ spans $L_{4.5}$%
}\hbox{%
$\infty222$ (shared)%
}%
}%
\hfill\copy\matricesbox
}%
\ifdim\wd\onelinebox>\myboxwidth
\hbox to \myboxwidth{%
$\Pi_{4,87}$ spans $L_{4.5}$%
\hfil
$\infty222$ (shared)%
}%
\box\matricesbox
\else
\hbox to \myboxwidth{%
\unhbox\onelinebox
}%
\fi
\else
\hbox to \myboxwidth{%
$\Pi_{4,87}$ spans $L_{4.5}$%
\hfil}%
\hbox to \myboxwidth{%
$\infty222$ (shared)%
\hfil}%
\box\matricesbox
\fi
}%
\hfill\discretionary{}{}{}%
}

\begin{thebibliography}{99}

\bibitem{Allcock-19}
Allcock, Daniel, Infinitely many hyperbolic Coxeter groups through
dimension 19, {\it Geom. Topol.} {\bf 10} (2006) 737--758.

\bibitem{Allcock-rk3}
Allcock, Daniel, The reflective Lorentzian lattices of rank~$3$, to
appear in {\it
  Mem. A.M.S.}; arXiv:1010.0486.

\bibitem{Allcock-rk3table}
Allcock, Daniel, Unabridged table of reflective lattices of rank 3,
arXiv:1111.1264.

\bibitem{Allcock-arxiv} Allcock, Daniel, Root systems for Lorentzian
  Kac-Moody algebras in rank 3 (arXiv
  version), arXiv:1209.0022.

\bibitem{Borcherds-fake-monster-Lie-algebra}
Borcherds, Richard E., The monster Lie algebra, {\it Adv. Math.} {\bf 83} (1990)30--47.

\bibitem{Borcherds-moonshine}
Borcherds, Richard E., Monstrous moonshine and monstrous Lie
superalgebras, {\it Invent. Math.} {\bf 109} (1992) 405--444.

\bibitem{Borcherds-automorphic-forms-with-singularities-on-Grassmannians}
Borcherds, Richard E., Automorphic forms with singularities on
Grassmannians, {\it Invent. Math.} {\bf 132} (1998) 491--562.

\bibitem{Carbone-etal}
Carbone, Lisa; Chung, Sjuvon; Cobbs, Leigh; McRae, Robert; Nandi,
Debajyoti; Naqvi, Yusra; and Penta, Diego: Classification of
hyperbolic Dynkin diagrams, root lengths and Weyl group orbits, {\it
  J. Physics A: Mathematical and Theoretical} {\bf 43} (2010) 1--30.

\bibitem{Esselmann} Esselmann, Frank, \"Uber die maximale Dimension
  von Lorentz-Gittern mit coendlicher Spiegelungsgruppe, {\it
    J. Number Theory} {\bf 61} (1996) 103--144.

\bibitem{Feingold-Frenkel}
Feingold, Alex J. and Frenkel, Igor B, A hyperbolic Kac-Moody algebra
and the theory of Siegel modular forms of genus $2$, {\it Math. Ann.}
{\bf 263} (1983) 87--144.

\bibitem{Gritsenko-Reflective-modular-forms-in-algebraic-geometry}
Gritsenko, Valery,
Reflective modular forms in algebraic geometry, 
arXiv:1005.3753.

\bibitem{GN-Igusa}
Gritsenko, V. A. and Nikulin, V. V., Igusa modular forms and ``the
simplest'' Lorentzian Kac-Moody algebras. {\it Mat. Sb.} {\bf 187}
(1996), no. 11, 27--66; English translation in {\it Sb. Math.} {\bf 187} (1996), no. 11,
1601--1641.

\bibitem{GN-Siegel-Automorphic-form-corrections-of-some-LKMLAs}
Gritsenko, Valeri A. and Nikulin, Viacheslav V., Siegel automorphic form
corrections of some Lorentzian Kac-Moody Lie
algebras, {\it Amer. J. Math.} {\bf 119} (1997) 181--224.

\bibitem{GN-K3-surfaces-Lorentzian-KMAs-and-Mirror-Symmetry}
Gritsenko, Valeri A. and Nikulin, Viacheslav V, $K3$
surfaces, Lorentzian Kac-Moody algebras and mirror
symmetry, {\it Math. Res. Lett.} {\bf 3} (1996) 211--229.

\bibitem{GN-AFaLKMAsI}
Gritsenko, Valeri A. and Nikulin, Viacheslav V., Automorphic forms and
Lorentzian Kac-Moody algebras I, {\it Internat. J. Math.} {\bf 9} (1998) 153--199.

\bibitem{GN-AFaLKMAsII}
Gritsenko, Valeri A. and Nikulin, Viacheslav V., Automorphic forms and
Lorentzian Kac-Moody algebras II. {\it Internat. J. Math.} {\bf 9} (1998)
201--275. 

\bibitem{GN-On-classification-of-LKMAs}
Gritsenko, V. A. and Nikulin, V. V., On the classification of
Lorentzian Kac-Moody algebras, {\it Uspekhi Mat. Nauk} {\bf 57}
(2002), no. 5(347), 79--138; English translation in {\it Russian
  Math. Surveys} {\bf 57} (2002), no. 5, 921--979.

\bibitem{Kac-first}
Kac, V. G.,
Simple graded Lie algebras of finite height, {\it
Funkcional. Anal. i Prilo\v{z}en} {\bf 1} (1967) no. 4, 82--83;
English translation in {\it Functional Anal. Appl.} {\bf 1} (1967)
no. 4, 328--329. 

\bibitem{Kac-book}
Kac, Victor G., {\it Infinite-dimensional Lie algebras} (third
edition). Cambridge University Press, Cambridge, 1990.

\bibitem{Moody-first}
Moody, R. V., Lie algebras associated with generalized Cartan
matrices, {\it Bull. Amer. Math. Soc.} {\bf 73} (1967) 217--221.

\bibitem{Nikulin-Algebraic-surfaces-with-log-terminal-singularities}
Nikulin, Viacheslav V., Basis of the diagram method for generalized
reflection groups in Lobachevsky spaces and algebraic surfaces with
nef anticanonical class, {\it Internat. J. Math.} {\bf 7} (1996) 71--108.

\bibitem{Nikulin-Reflection-groups-in-hyperbolic-spaces-and-the-denominator-formula-for-LKMLAs}
Nikulin, V. V., Reflection groups in Lobachevski\v{i} spaces and an
identity for the denominator of Lorentzian Kac-Moody
algebras, {\it Izv. Ross. Akad. Nauk Ser. Mat.} {\bf 60} (1996) no. 2,
73--106; English translation in {\it Izv. Math.} {\bf 60} (1996)
no. 2, 305--334.

\bibitem{Nikulin-rank-3} Nikulin, V. V., On the classification of
  hyperbolic root systems of rank three. (Russian) {\it
    Tr. Mat. Inst. Steklova} {\bf 230} (2000), 256 pp.; english
  translation in {\it Proc. Steklov Inst. Math.} {\bf 230} (2000),
  no. 3, 1--241.

\bibitem{Nikulin-A-theory-of-Lorentzian-KMAs}
Nikulin, V. V., A theory of Lorentzian Kac-Moody algebras,
{\it J. Math. Sci.} {\bf 106} (2001) 3212--3221.

\bibitem{Saclioglu}
Sa\c{c}lio\u{g}lu, Cihan, Dynkin diagrams for hyperbolic Kac-Moody
algebras, {\it J. Physics A} {\bf 22} (1989) 3753--3769.

\end{thebibliography}
\end{document}